%% file: okamoto.tex
\documentclass[reqno]{amsart}
\usepackage[margin=1in]{geometry}
\usepackage{xcolor,colortbl,diagbox,comment,graphicx,mathrsfs,mathpazo,amssymb}
\usepackage{MnSymbol}
\usepackage{mathtools}
\usepackage{multirow,array}
\usepackage{makecell}

\usepackage{caption}
\usepackage[skip=0.5ex]{subcaption}

\def\eq{\begin{equation}}
\def\endeq{\end{equation}}
\def\bbm{\begin{bmatrix}}
\def\ebm{\end{bmatrix}}
\def\bpm{\begin{pmatrix}}
\def\epm{\end{pmatrix}}
\def\bvm{\begin{vmatrix}}
\def\evm{\end{vmatrix}}

\newtheorem{rhp}{Riemann-Hilbert Problem}
\newtheorem{proposition}{Proposition}
\newtheorem{theorem}{Theorem}
\newtheorem{corollary}{Corollary}
\newtheorem{lemma}{Lemma}
\theoremstyle{definition}
\newtheorem{definition}{Definition}
\newtheorem{remark}{Remark}
\newtheorem{example}{Example}
\numberwithin{equation}{section}
\makeatletter
\@addtoreset{equation}{section}
\makeatother

\makeatletter
\renewcommand*\env@matrix[1][c]{\hskip -\arraycolsep
  \let\@ifnextchar\new@ifnextchar
  \array{*\c@MaxMatrixCols #1}}
\makeatother

\newcommand{\bo}{\mathcal{O}}
\newcommand{\ee}{\mathrm{e}}
\newcommand{\ii}{\mathrm{i}}
\newcommand{\dd}{\mathrm{d}}
\newcommand{\critclosure}{K}
\newcommand{\circledomain}{\mathcal{D}_\circ}

\newcommand{\rectangle}{\mathcal{B}_\medsquare}
\newcommand{\TR}{\mathcal{B}_\medtriangleright}
\newcommand{\TI}{\mathcal{B}_\medtriangleup}
\newcommand{\OkamotoExterior}{\mathcal{E}_\mathrm{gO}}
\newcommand{\HermiteExterior}{\mathcal{E}_\mathrm{gH}}

\newcommand{\tallstrut}{$\vphantom{\Bigg[\Bigg]}$}
\newcommand{\shortstrut}{$\vphantom{\Big[\Big]}$}

\newcommand{\tw}{{\scriptsize\smash{\substack{\rcurvearrownw \\[-0.7ex]\rcurvearrowse}}}}

\newcommand{\PreserveBackslash}[1]{\let\temp=\\#1\let\\=\temp}
\newcolumntype{C}[1]{>{\PreserveBackslash\centering}p{#1}}
\newcolumntype{R}[1]{>{\PreserveBackslash\raggedleft}p{#1}}
\newcolumntype{L}[1]{>{\PreserveBackslash\raggedright}p{#1}}

\title[Rational Solutions of Painlev\'e-IV]{Large-degree asymptotics of rational Painlev\'e-IV solutions by the isomonodromy method}
\author{Robert J. Buckingham}
\address[R. J. Buckingham]{Department of Mathematical Sciences\\ University of Cincinnati\\ PO Box 210025\\ Cincinnati, OH 45221.}
\email{buckinrt@uc.edu}
\urladdr{http://homepages.uc.edu/~buckinrt/}
\author{Peter D. Miller}
\address[P. D. Miller]{Department of Mathematics\\University of Michigan\\East Hall\\ 530 Church St.\\Ann Arbor, MI 48109.}
\email{millerpd@umich.edu}
\urladdr{http://www.math.lsa.umich.edu/~millerpd/}
\begin{document}
\begin{abstract}
%The Painlev\'e-IV equation has a family of rational solutions generated 
%by the generalized Okamoto polynomials indexed by two positive integers $m$ and $n$.  
%Numerical plots suggest the zeros and poles form deformed triangles attached 
%to the edges of a a deformed $n\times m$ rectangular 
%grid.  
The Painlev\'e-IV equation has two families of rational solutions generated respectively by the generalized Hermite polynomials and the generalized Okamoto polynomials.  We apply the isomonodromy method to represent all of these rational solutions by means of two related Riemann-Hilbert problems, each of which involves two integer-valued parameters related to the two parameters in the Painlev\'e-IV equation.  We then use the steepest-descent method to analyze the rational solutions in the limit that at least one of the parameters is large.  Our analysis provides rigorous justification for formal asymptotic arguments that suggest that in general solutions of Painlev\'e-IV with large parameters behave either as an algebraic function or an elliptic function.  Moreover, the results show that the elliptic approximation holds on the union of a curvilinear rectangle and, in the case of the generalized Okamoto rational solutions, four curvilinear triangles each of which shares an edge with the rectangle; the algebraic approximation is valid in the complementary unbounded domain.  We compare the theoretical predictions for the locations of the poles and zeros with numerical plots of the actual poles and zeros obtained from the generating polynomials, and find excellent agreement.
\end{abstract}



\maketitle

\tableofcontents

\input{OkamotoIntroEtc}

\input{OkamotoAsymp}

\appendix

\input{OkamotoAppendices}

\end{document}

%% file: OkamotoIntroEtc.tex
\section{Introduction}
%The Painlev\'e-IV equation
% for a function $u=u(x)$:
%\eq
%u''=\frac{(u')^2}{2u}+\frac{3}{2}u^3+4xu^2+2(x^2+1-2\Theta_\infty)u-\frac{8\Theta_0^2}{u}, \quad u:\mathbb{C}\to\mathbb{C} \text{ with parameters } \Theta_0,\Theta_\infty\in\mathbb{C}
%\label{p4}
%\endeq
%is a fundamental equation of mathematical physics with applications ranging 
%from nonlinear wave equations \cite{BassomCH:1996} and quantum gravity 
%\cite{FokasIK:1991} to orthogonal polynomials \cite{DaiK:2009} and random 
%matrix theory \cite{ChenF:2006,ForresterW:2001,OsipovSZ:2010}.  
\input{BriefIntro.tex}

\subsection{Rational solutions of Painlev\'e-IV}
\label{sec:rational-PIV-solutions}
The Painlev\'e-IV equation in the form \eqref{p4} has a rational solution if the parameters satisfy either $(\Theta_0,\Theta_\infty)\in\Lambda_\mathrm{gH} = \Lambda_\mathrm{gH}^{[1]-}\sqcup\Lambda_\mathrm{gH}^{[2]-}\sqcup\Lambda_\mathrm{gH}^{[3]+}$ (the generalized Hermite family), where
\eq
\begin{split}
\Lambda_\mathrm{gH}^{[1]-}&:=\left\{(\Theta_0,\Theta_\infty)\in\mathbb{C}^2:  (\Theta_0+\Theta_\infty,\Theta_0-\Theta_\infty)\in\mathbb{Z}^2, \Theta_0<0, \Theta_\infty>-\Theta_0\right\}\\
\Lambda_\mathrm{gH}^{[2]-}&:=\left\{(\Theta_0,\Theta_\infty)\in\mathbb{C}^2:  (\Theta_0+\Theta_\infty,\Theta_0-\Theta_\infty)\in\mathbb{Z}^2, \Theta_0<0, \Theta_\infty\le \Theta_0\right\}\\
\Lambda_\mathrm{gH}^{[3]+}&:=\left\{(\Theta_0,\Theta_\infty)\in\mathbb{C}^2:  (\Theta_0+\Theta_\infty,\Theta_0-\Theta_\infty)\in\mathbb{Z}^2, \Theta_0>0, -\Theta_0<\Theta_\infty\le\Theta_0\right\},
\end{split}
%(\Theta_0,\Theta_\infty)\in\Lambda_{\mathrm{gH}}:=\left\{(\Theta_0,\Theta_\infty)\in\mathbb{C}^2: 
%((\Theta_0-\tfrac{1}{2})+(\Theta_\infty-\tfrac{1}{2}),(\Theta_0-\tfrac{1}{2})-(\Theta_\infty-\tfrac{1}{2}))\in \mathbb{Z}^2_{m+n\ge 0}\right\}
\label{eq:gH-lattice}
\endeq
%where $(m,n)\in\mathbb{Z}^2_{m+n\ge 0}\subset\mathbb{Z}^2$ if and only if $m+n\ge 0$,
or $(\Theta_0,\Theta_\infty)\in\Lambda_{\mathrm{gO}}$ (the generalized Okamoto family), where
\eq
%(\Theta_0,\Theta_\infty)\in
\Lambda_{\mathrm{gO}}:=\left\{(\Theta_0,\Theta_\infty)\in\mathbb{C}^2: 
%((\Theta_0-\tfrac{1}{6})+(\Theta_\infty-\tfrac{1}{2}),(\Theta_0-\tfrac{1}{6})-(\Theta_\infty-\tfrac{1}{2}))
(\Theta_0+\Theta_\infty-\tfrac{2}{3},\Theta_0-\Theta_\infty + \tfrac{1}{3})
\in\mathbb{Z}^2\right\}.
\label{eq:gO-lattice}
\endeq
Note that $\Lambda_\mathrm{gH}\cap\Lambda_\mathrm{gO}=\emptyset$.  These sufficient conditions fail to be necessary\footnote{We discount the case $\Theta_0=0$ for which \eqref{p4} multiplied by $u(x)$ admits the zero solution $u(x)\equiv 0$.} only because whenever \eqref{p4} admits a rational solution for parameters $(\Theta_0,\Theta_\infty)\in\Lambda_\mathrm{gH}\sqcup\Lambda_\mathrm{gO}$ then it obviously admits the same rational solution also for parameters $(-\Theta_0,\Theta_\infty)$, which do not yield a point of either $\Lambda_\mathrm{gH}$ or $\Lambda_\mathrm{gO}$, because the Painlev\'e-IV equation only involves $\Theta_0^2$.  
The sets $\Lambda_\mathrm{gH}$ and $\Lambda_\mathrm{gO}$ are in one-to-one correspondence under the mapping $(\Theta_0,\Theta_\infty)\mapsto(\Theta_0^2,\Theta_\infty)$ with their images, which are sets of actual parameters of the Painlev\'e-IV equation; modulo the kernel of this mapping, the condition $(\Theta_0,\Theta_\infty)\in\Lambda_\mathrm{gH}\sqcup\Lambda_\mathrm{gO}$ is also necessary for existence of a rational solution of \eqref{p4}.  The sets $\Lambda_\mathrm{gH}$ and $\Lambda_\mathrm{gO}$ are shown in Figure~\ref{fig:ThetasPlane}.
\begin{figure}[h]
\begin{center}
\includegraphics{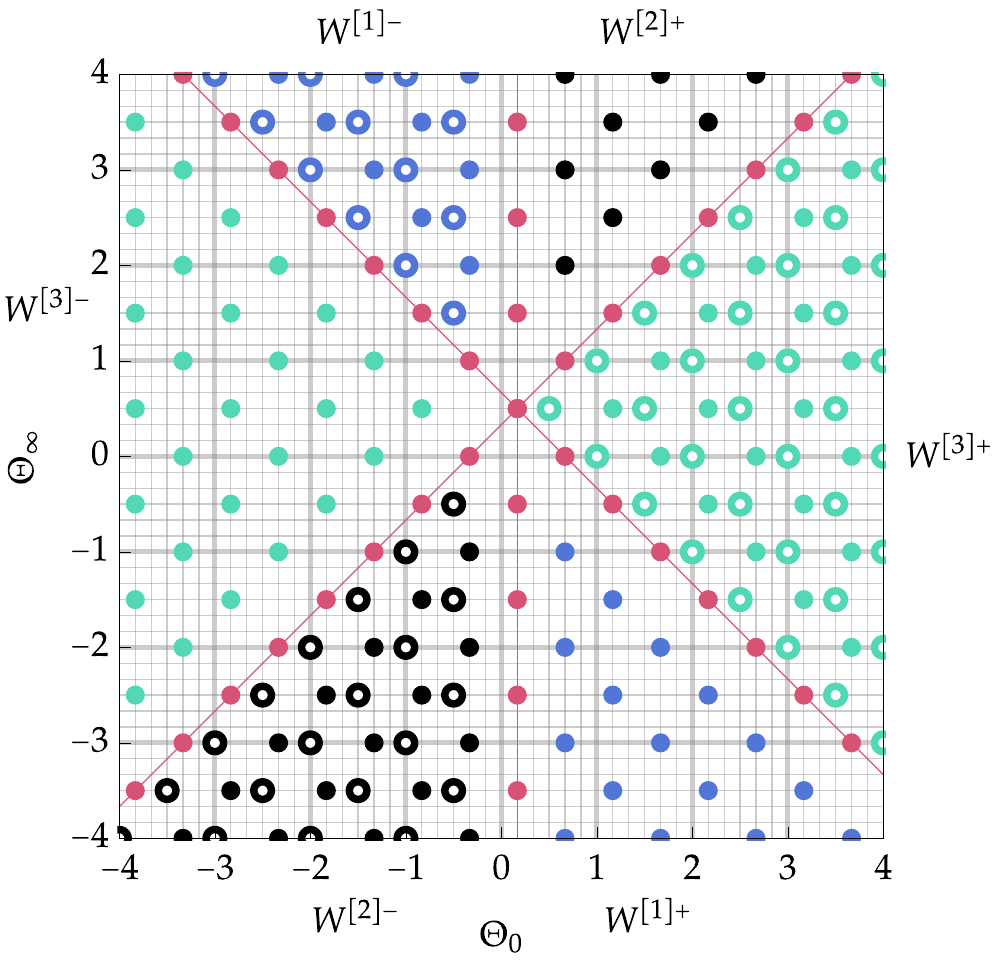}
\end{center}
\caption{Unfilled dots:  the points in the set $\Lambda_\mathrm{gH}$.  Filled dots:  the points in the set $\Lambda_\mathrm{gO}$.   Also shown are the disjoint open sectors $W^{[1]\pm}$, $W^{[2]\pm}$, and $W^{[3]\pm}$.  The complementary set $L$ is the union of the three red lines intersecting at $(\Theta_0,\Theta_\infty)=(\tfrac{1}{6},\tfrac{1}{2})$.}
\label{fig:ThetasPlane}
\end{figure}

By substituting appropriate Taylor or Laurent series into \eqref{p4} one easily sees that for any solution $u(x)$, all poles must be simple with residue $\pm 1$ and for $\Theta_0\neq 0$ all zeros $x=x_0$ must be simple with $u'(x_0)=\pm 4\Theta_0$. For rational solutions, the latter can be written as $u'(x_0)=\pm 4|\Theta_0|$ because $\Theta_0\in\mathbb{R}$.

\subsubsection{Representation in terms of special polynomials}
\label{sec:special-polynomials}
For given $(\Theta_0,\Theta_\infty)\in\Lambda_\mathrm{gH}\cup\Lambda_\mathrm{gO}$, the rational solution of \eqref{p4} is unique \cite{Gromak:1987,Lukashevich:1967,Murata:1985}.
This unique rational solution can be constructed by iteration of elementary B\"acklund transformations such as those described at the beginning of Section~\ref{sec:intro-Baecklund} below (see also Section~\ref{sec:sub-Baecklund} of Appendix~\ref{app:Isomonodromy}); such transformations allow one to generate the whole parameter space step-by-step, starting from a given point where the rational solution is known.  In general, B\"acklund transformations can be applied to any solution of \eqref{p4}, however since the solutions of interest are rational, the iterated B\"acklund transformations can also be represented as recurrence relations for certain special polynomial factors, as we will now explain.
First, suppose that $(\Theta_0,\Theta_\infty)\in\Lambda_\mathrm{gH}$.
It is natural to associate each of the three disjoint components of $\Lambda_\mathrm{gH}$ with a different \emph{type} $j=1,2,3$ of rational solution in the gH family.  Following \cite{Clarkson:2006} we index the points of the three disjoint components of $\Lambda_\mathrm{gH}$ 
as follows:
\eq
\begin{split}
\Lambda_\mathrm{gH}^{[1]-} &= \left\{(\Theta_0,\Theta_\infty)=(\Theta_{0,\mathrm{gH}}^{[1]}(m,n),\Theta_{\infty,\mathrm{gH}}^{[1]}(m,n)):  (m,n)\in\mathbb{Z}_{\ge 0}\times\mathbb{Z}_{>0}\right\}\\
\Lambda_\mathrm{gH}^{[2]-} &=  \left\{(\Theta_0,\Theta_\infty)=(\Theta_{0,\mathrm{gH}}^{[2]}(m,n),\Theta_{\infty,\mathrm{gH}}^{[2]}(m,n)):  (m,n)\in\mathbb{Z}_{> 0}\times\mathbb{Z}_{\ge 0}\right\},\\
\Lambda_\mathrm{gH}^{[3]+} &=  \left\{(\Theta_0,\Theta_\infty)=(\Theta_{0,\mathrm{gH}}^{[3]}(m,n),\Theta_{\infty,\mathrm{gH}}^{[3]}(m,n)):  (m,n)\in\mathbb{Z}_{\ge 0}\times\mathbb{Z}_{\ge 0}\right\},
\end{split}
\endeq
in which the parameters $(\Theta_{0,\mathrm{gH}}^{[j]}(m,n),\Theta_{\infty,\mathrm{gH}}^{[j]}(m,n))$ for each type are as indicated in Table~\ref{tab:gH}.
\begin{table}[h]
\caption{Representation of gH solutions of Painlev\'e-IV in terms of gH polynomials.}
\begin{tabular}{|l|c|c|c|c|}
\hline
Type $j$ & \tallstrut $\Theta_{0,\mathrm{gH}}^{[j]}(m,n)$ & \tallstrut $\Theta_{\infty,\mathrm{gH}}^{[j]}(m,n)$ & \tallstrut $\tau_\mathrm{gH}^{[j]}(x;m,n)$ & \tallstrut $u_\mathrm{gH}^{[j]}(x;m,n)$ \\
\hline\hline
$1$ & $-\tfrac{1}{2}n$ & $1+m+\tfrac{1}{2}n$ & \tallstrut $\displaystyle \frac{H_{m+1,n}(x)}{H_{m,n}(x)}$ & \tallstrut $\displaystyle 2n\frac{H_{m,n+1}(x)H_{m+1,n-1}(x)}{H_{m,n}(x)H_{m+1,n}(x)}$\\
\hline
$2$ & $-\tfrac{1}{2}m$ & $-\tfrac{1}{2}m-n$ & \tallstrut $\displaystyle \frac{H_{m,n}(x)}{H_{m,n+1}(x)}$ & \tallstrut $\displaystyle -2m\frac{H_{m-1,n+1}(x)H_{m+1,n}(x)}{H_{m,n}(x)H_{m,n+1}(x)}$\\
\hline
$3$ & $\tfrac{1}{2}+\tfrac{1}{2}(m+n)$ & $\tfrac{1}{2}+\tfrac{1}{2}(n-m)$ & \tallstrut $\displaystyle \ee^{-x^2}\frac{H_{m,n+1}(x)}{H_{m+1,n}(x)}$ & \tallstrut $\displaystyle -\frac{H_{m,n}(x)H_{m+1,n+1}(x)}{H_{m,n+1}(x)H_{m+1,n}(x)}$\\
\hline
\end{tabular}
\label{tab:gH}
\end{table}
The \emph{generalized Hermite (gH) polynomials} $\{H_{m,n}(x)\}_{(m,n)\in\mathbb{Z}_{\ge 0}\times\mathbb{Z}_{\ge 0}}$ are determined from the recurrence relations
\eq
\begin{split}
2mH_{m+1,n}(x)H_{m-1,n}(x)&=H_{m,n}(x)H_{m,n}''(x)-H'_{m,n}(x)^2+2mH_{m,n}(x)^2\\
2nH_{m,n+1}(x)H_{m,n-1}(x)&=-H_{m,n}(x)H_{m,n}''(x)+H_{m,n}'(x)^2+2nH_{m,n}(x)^2
\end{split}
\endeq
with initial conditions $H_{0,0}(x)=H_{1,0}(x)=H_{0,1}(x)=1$ and $H_{1,1}(x)=2x$.  In terms of these,
the unique rational solution $u(x)=u^{[j]}_{\mathrm{gH}}(x;m,n)$ of \eqref{p4} associated to the parameters $(\Theta_0,\Theta_\infty)=(\Theta_{0,\mathrm{gH}}^{[j]}(m,n),\Theta_{\infty,\mathrm{gH}}^{[j]}(m,n))$ can be written in logarithmic derivative form
\eq
u_\mathrm{gH}^{[j]}(x;m,n)=\frac{\dd}{\dd x}\log \left(\tau_\mathrm{gH}^{[j]}(x;m,n)\right)
\label{eq:gH-tau}
\endeq
for a suitable function $\tau_\mathrm{gH}^{[j]}(x;m,n)$ expressed in terms of a ratio of two gH polynomials, or alternately $u_\mathrm{gH}^{[j]}(x;m,n)$ can be written as a ratio of four gH polynomials, as shown in Table~\ref{tab:gH}.
$H_{m,n}(x)$ has degree $mn$, so it follows from \eqref{eq:gH-tau} that 
\eq
\begin{split}
u_\mathrm{gH}^{[1]}(x;m,n)&=\frac{n}{x}(1+o(1)) = -\frac{2\Theta_0}{x}(1+o(1)),\quad x\to\infty,\quad\Theta_0=\Theta_{0,\mathrm{gH}}^{[1]}(m,n),\\
u_\mathrm{gH}^{[2]}(x;m,n)&=-\frac{m}{x}(1+o(1)) = \frac{2\Theta_0}{x}(1+o(1)),\quad x\to\infty,\quad
\Theta_0=\Theta_{0,\mathrm{gH}}^{[2]}(m,n),\\
u_\mathrm{gH}^{[3]}(x;m,n)&=-2x(1+o(1)),\quad x\to\infty.
\end{split}
\label{eq:gHsolutionsLarge-x}
\endeq
See \cite{Buckingham18,Clarkson:2003,NoumiY:1999} for more details.  

To construct the rational solutions for $(\Theta_0,\Theta_\infty)\in\Lambda_\mathrm{gO}$, it is convenient to first divide the real $(\Theta_0,\Theta_\infty)$-plane into a disjoint union of six open sectors:
\eq
\begin{split}
W^{[1]\pm}&:=\left\{(\Theta_0,\Theta_\infty)\in\mathbb{R}^2:  \pm(\Theta_0-\tfrac{1}{6})>0, \pm(\Theta_0+\Theta_\infty-\tfrac{2}{3})<0\right\}\\
W^{[2]\pm}&:=\left\{(\Theta_0,\Theta_\infty)\in\mathbb{R}^2:  \pm(\Theta_0-\tfrac{1}{6})>0, =\pm(\Theta_\infty-\Theta_0-\tfrac{1}{3})>0\right\}\\
W^{[3]\pm}&:=\left\{(\Theta_0,\Theta_\infty)\in\mathbb{R}^2:  \pm(\Theta_0-\tfrac{1}{6})>0, |\Theta_\infty-\tfrac{1}{2}|<|\Theta_\infty-\tfrac{1}{6}|\right\},
\end{split}
\label{eq:gO-half-sectors}
\endeq
and the complement $L$ in $\mathbb{R}^2$ of their union; $L$ is the union of three lines through the point $(\tfrac{1}{6},\tfrac{1}{2})$:  $\Theta_\infty-\tfrac{1}{2}=\pm(\Theta_0-\tfrac{1}{6})$ and $\Theta_0=\tfrac{1}{6}$.  Note that $L$ contains points of $\Lambda_\mathrm{gO}$ but no points of $\Lambda_\mathrm{gH}$.  In fact, $\Lambda_\mathrm{gH}^{[1]-}\subset W^{[1]-}$, $\Lambda_\mathrm{gH}^{[2]-}\subset W^{[2]-}$, and $\Lambda_\mathrm{gH}^{[3]+}\subset W^{[3]+}$.  
See Figure~\ref{fig:ThetasPlane}.  Following a similar indexing scheme as for the components of $\Lambda_\mathrm{gH}$, we can describe the points of $\Lambda_\mathrm{gO}\setminus L$ and the associated rational solutions of \eqref{p4}.  First, consider the points in the same three sectors where the gH rational solutions live (we deal with the remainder of the lattice $\Lambda_\mathrm{gO}$ in the last paragraph of Section~\ref{sec:special-polynomials}).
\begin{table}[h]
\caption{Representation of gO solutions of Painlev\'e-IV in terms of gO polynomials.}
\begin{tabular}{|l|c|c|c|c|}
\hline
Type $j$ & \tallstrut $\Theta_{0,\mathrm{gO}}^{[j]}(m,n)$ & \tallstrut $\Theta_{\infty,\mathrm{gO}}^{[j]}(m,n)$ & \tallstrut $\tau_\mathrm{gO}^{[j]}(x;m,n)$ & \tallstrut $u_\mathrm{gO}^{[j]}(x;m,n)$ \\
\hline\hline
$1$ & $\tfrac{1}{6}-\tfrac{1}{2}n$ & $\tfrac{1}{2}+m+\tfrac{1}{2}n$ & \tallstrut $\displaystyle \ee^{-\frac{1}{3}x^2}\frac{Q_{m+1,n}(x)}{Q_{m,n}(x)}$ & \tallstrut $\displaystyle -\frac{\sqrt{2}}{3}\frac{Q_{m,n+1}(x)Q_{m+1,n-1}(x)}{Q_{m,n}(x)Q_{m+1,n}(x)}$\\
\hline
$2$ & $\tfrac{1}{6}-\tfrac{1}{2}m$ & $\tfrac{1}{2}-\tfrac{1}{2}m-n$ & \tallstrut $\displaystyle \ee^{-\frac{1}{3}x^2}\frac{Q_{m,n}(x)}{Q_{m,n+1}(x)}$ & \tallstrut $\displaystyle -\frac{\sqrt{2}}{3}\frac{Q_{m-1,n+1}(x)Q_{m+1,n}(x)}{Q_{m,n}(x)Q_{m,n+1}(x)}$\\
\hline
$3$ & $\tfrac{1}{6}+\tfrac{1}{2}(m+n)$ & $\tfrac{1}{2}+\tfrac{1}{2}(n-m)$ & \tallstrut $\displaystyle \ee^{-\frac{1}{3}x^2}\frac{Q_{m,n+1}(x)}{Q_{m+1,n}(x)}$ & \tallstrut $\displaystyle -\frac{\sqrt{2}}{3}\frac{Q_{m,n}(x)Q_{m+1,n+1}(x)}{Q_{m,n+1}(x)Q_{m+1,n}(x)}$\\
\hline
\end{tabular}
\label{tab:gO}
\end{table}
We index lattice points in the sectors $W^{[1]-}$, $W^{[2]-}$, and $W^{[3]+}$ by certain integers $(m,n)$ as follows:
\eq
\begin{split}
\Lambda_\mathrm{gO}\cap W^{[1]-} &=  \left\{(\Theta_0,\Theta_\infty)=(\Theta_{0,\mathrm{gO}}^{[1]}(m,n),\Theta_{\infty,\mathrm{gO}}^{[1]}(m,n)):  (m,n)\in\mathbb{Z}_{>0}\times\mathbb{Z}_{>0}\right\},\\
\Lambda_\mathrm{gO}\cap W^{[2]-} &=  \left\{(\Theta_0,\Theta_\infty)=(\Theta_{0,\mathrm{gO}}^{[2]}(m,n),\Theta_{\infty,\mathrm{gO}}^{[2]}(m,n)):  (m,n)\in\mathbb{Z}_{>0}\times\mathbb{Z}_{>0}\right\},\\
\Lambda_\mathrm{gO}\cap W^{[3]+} &=\left\{(\Theta_0,\Theta_\infty)=(\Theta_{0,\mathrm{gO}}^{[3]}(m,n),\Theta_{\infty,\mathrm{gO}}^{[3]}(m,n)):  (m,n)\in\mathbb{Z}_{>0}\times\mathbb{Z}_{>0}\right\}.
\end{split}
\endeq
The \emph{generalized Okamoto (gO) polynomials} $\{Q_{m,n}(x)\}_{(m,n)\in\mathbb{Z}^2}$ are defined by the recurrence relations
\eq
\begin{split}
Q_{m+1,n}(x)Q_{m-1,n}(x)&=\tfrac{9}{2}[Q_{m,n}(x)Q_{m,n}''(x)-Q_{m,n}'(x)^2]+[2x^2+3(2m+n-1)]Q_{m,n}(x)^2\\
Q_{m,n+1}(x)Q_{m,n-1}(x)&=\tfrac{9}{2}[Q_{m,n}(x)Q_{m,n}''(x)-Q_{m,n}'(x)^2]+[2x^2+3(1-m-2n)]Q_{m,n}(x)^2
\end{split}
\label{eq:gOrecurrence}
\endeq
with the initial conditions $Q_{0,0}(x)=Q_{1,0}(x)=Q_{0,1}(x)=1$ and $Q_{1,1}(x)=\sqrt{2}x$.
The unique rational solution $u(x)=u_\mathrm{gO}^{[j]}(x;m,n)$ of \eqref{p4} for parameters $(\Theta_0,\Theta_\infty)=(\Theta_{0,\mathrm{gO}}^{[j]}(m,n),\Theta_{\infty,\mathrm{gO}}^{[j]}(m,n))$ can then be expressed in terms of these polynomials either by a logarithmic derivative formula analogous to \eqref{eq:gH-tau}
\eq
u_\mathrm{gO}^{[j]}(x)=\frac{\dd}{\dd x}\log\left(\tau_\mathrm{gO}^{[j]}(x;m,n)\right)
\label{eq:gO-tau}
\endeq
or directly as a ratio of four polynomials as shown in Table~\ref{tab:gO}.
If both indices are positive, $Q_{m,n}(x)$ has degree $mn+m(m-1)+n(n-1)$ which exceeds the degree of $H_{m,n}(x)$ by twice the sum of two triangular numbers.  The gO polynomials were first studied by Noumi and Yamada \cite{NoumiY:1999}.  Note that, unlike the gH polynomials, the gO polynomials can be defined for negative $m$ and/or $n$.  From the recurrence relations \eqref{eq:gOrecurrence} it follows easily that 
\eq
Q_{m,n}(x) = Q_{n,1-m-n}(x) = Q_{1-m-n,m}(x), \quad (m,n)\in\mathbb{Z}^2.
\endeq
In particular, gO polynomials with indices of opposite signs can be easily identified with gO polynomials with either both nonnegative or both nonpositive indices.

The formul\ae\ \eqref{eq:gH-tau} and \eqref{eq:gO-tau} expressing
%\eqref{eq:gH-1-solution}, \eqref{eq:gH-2-solution}, \eqref{eq:gH-3-solution}, \eqref{eq:gO-1m-solution}, \eqref{eq:gO-2m-solution}, and \eqref{eq:gO-3p-solution} for 
$u(x)$ in terms of logarithmic derivatives of ratios of special polynomials are common in the literature (see, e.g., \cite{Clarkson:2006}) and they lead quickly to the asymptotic relations \eqref{eq:gHsolutionsLarge-x} and likewise show that 
\eq
u(x)=-\tfrac{2}{3}x(1+o(1)),\quad x\to\infty,\quad \text{$u(x)$ a rational solution of \eqref{p4} for $(\Theta_0,\Theta_\infty)\in\Lambda_\mathrm{gO}$}.
\label{eq:gOsolutionsLarge-x}
\endeq  
On the other hand, the alternate formul\ae\ for $u(x)$ as ratios of four special polynomials are not as well-known but can be obtained by combining the implied large-$x$ asymptotics with the zero and pole locations for rational solutions tabulated in \cite[Table 2]{MasoeroR18} and the known leading coefficients of the four polynomial factors.

Now we discuss how to represent the rational solutions for the rest of the lattice $\Lambda_\mathrm{gO}$ in terms of special polynomials.  One important difference between the gH and gO rational solutions of \eqref{p4} is that while the gH solution formul\ae\ 
%\eqref{eq:gH-1-parameters}--\eqref{eq:gH-3-solution} 
in Table~\ref{tab:gH}
are only valid for the indicated ranges of nonnegative values of $(m,n)$, each of the gO solution formul\ae\ 
%\eqref{eq:gO-1m-parameters}--\eqref{eq:gO-3p-solution} 
in Table~\ref{tab:gO}
actually defines a rational solution of \eqref{p4} for all $(m,n)\in\mathbb{Z}^2$, and hence each reproduces the unique rational solution of \eqref{p4} for \emph{any} point of the gO parameter lattice $\Lambda_\mathrm{gO}$.  In other words, there are three different ways to express every gO rational solution in terms of different gO polynomials.  While on one hand this fact shows that the distinction of three types of gO rational solutions is somewhat artificial, an immediate application is that we can define the rational solution of \eqref{p4} for $(\Theta_0,\Theta_\infty)\in\Lambda_\mathrm{gO}\cap (W^{[1]+}\cup W^{[2]+}\cup W^{[3]-})$ simply by replacing the condition $(m,n)\in\mathbb{Z}_{>0}\times\mathbb{Z}_{>0}$ with $(m,n)\in\mathbb{Z}_{<0}\times\mathbb{Z}_{<0}$ in 
\eqref{eq:gO-half-sectors} and using the same solution formul\ae\ given in Table~\ref{tab:gO} but with negative indices on the gO polynomials.
%\eqref{eq:gO-1m-parameters}--\eqref{eq:gO-1m-solution} or \eqref{eq:gO-2m-parameters}--\eqref{eq:gO-2m-solution} or \eqref{eq:gO-3p-parameters}--\eqref{eq:gO-3p-solution}, respectively.  
Likewise, any of these formul\ae\ can be extended to $m=0$ or $n=0$ to yield representations of the rational solutions of \eqref{p4} for $(\Theta_0,\Theta_\infty)\in\Lambda_\mathrm{gO}\cap L$ (the red dots in Figure~\ref{fig:ThetasPlane}); these formul\ae\ necessarily involve either $Q_{0,n}(x)$ or $Q_{m,0}(x)$, special cases of $Q_{m,n}(x)$ that are simply called Okamoto polynomials and were first studied in \cite{Okamoto:1986} (likewise $H_{m,1}(x)$ are the classical Hermite polynomials).  

\subsubsection{Applications and qualitative properties of gH and gO polynomials and the associated rational functions}
\label{sec:rational-PIV-applications}
The gH and gO polynomials are of independent interest, 
and both families of polynomials are notable for their 
remarkably rich mathematical structure.  
Applications of these special 
polynomials and the rational solutions of \eqref{p4} generated from them include rational-oscillatory solutions of the 
defocusing nonlinear Schr\"odinger equation \cite{Clarkson:2006}, rational 
solutions of the Boussinesq equation \cite{Clarkson:2008} and the classical 
Boussinesq system \cite{Clarkson:2009}, rational-logarithmic solutions of the 
dispersive water wave equation and the modified Boussinesq equation 
\cite{ClarksonT:2009}, the point vortex equations with quadrupole background 
flow \cite{Clarkson:2009b}, the steady-state distribution of electric charges 
in a parabolic potential \cite{MarikhinSBP:2000}, and rational extensions of 
the harmonic oscillator and related exceptional orthogonal polynomials 
\cite{MarquetteQ:2016}.

From plots \cite{Clarkson:2003,NoumiY:1999}, one can see that the $mn$ zeros of the gH polynomial $H_{m,n}(x)$ are arranged in the 
complex $x$-plane in a quasi-rectangular $m\times n$ grid.  Likewise, if both $m$ and $n$ are positive, the zeros of the gO polynomial $Q_{m,n}(x)$ are arranged in a quasi-rectangular $m\times n$ grid, two quasi-triangular grids with a base of $m-1$ zeros, and two quasi-triangular grids with a base of $n-1$ zeros.  If both $m$ and $n$ are negative, then one has instead a $|n|\times |m|$ quasi-rectangular grid, two quasi-triangular grids with a base of $|n|$ zeros, and two quasi-triangular grids with a base of $|m|$ zeros.  See Figure~\ref{fig:rootplots}.  Note that, despite appearance of a qualitatively similar quasi-rectangular grid of the same dimensions in all three cases, in general $H_{m,n}(x)$ is not a factor of either $Q_{m,n}(x)$ or $Q_{-n,-m}(x)$.
\begin{figure}[h]
\begin{center}
\includegraphics{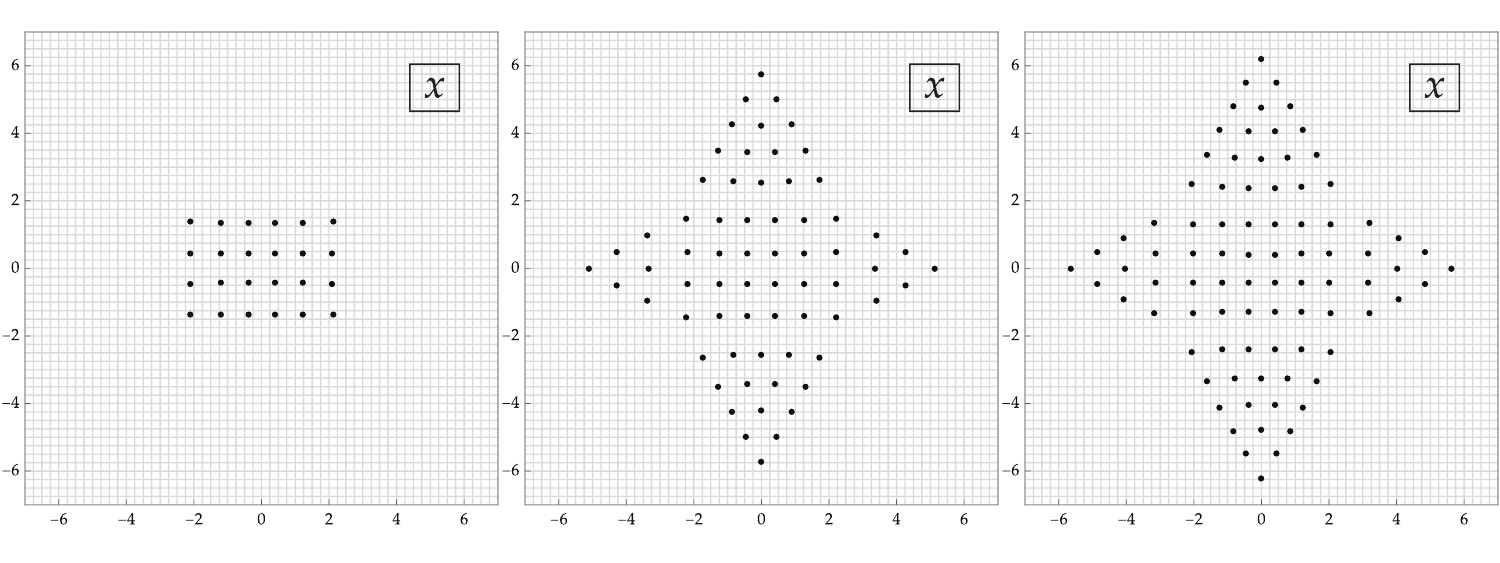}
\end{center}
\caption{Left:  the roots of $H_{6,4}(x)$ in the complex $x$-plane.  Center:  the roots of $Q_{6,4}(x)$ in the complex $x$-plane.  Right:  the roots of $Q_{-4,-6}(x)$ in the complex $x$-plane.}
\label{fig:rootplots}
\end{figure}
Since $H_{m,1}(x)$ is a classical Hermite polynomial of degree $m$, in this case the roots are exactly real and the quasi-rectangle degenerates to a line.  Likewise for $Q_{m,1}(x)$ one has a (non-generalized) Okamoto polynomial for which the quasi-rectangle again degenerates to a line and two of the quasi-triangles degenerate to points, while two quasi-triangles of base $m$ remain.  

Assembling the rational solutions from the gH polynomials using Table~\ref{tab:gH}, one can easily display the interaction between the poles and zeros contributed by different gH polynomial factors, and illustrate how the organization of the poles and zeros varies with the parameters in $\Lambda_\mathrm{gH}$, as shown in Figure~\ref{fig:theta-map-gH}.
\begin{figure}[h]
\hspace{-0.2in}
\begin{tabular}{c}
\includegraphics[height=1.2in]{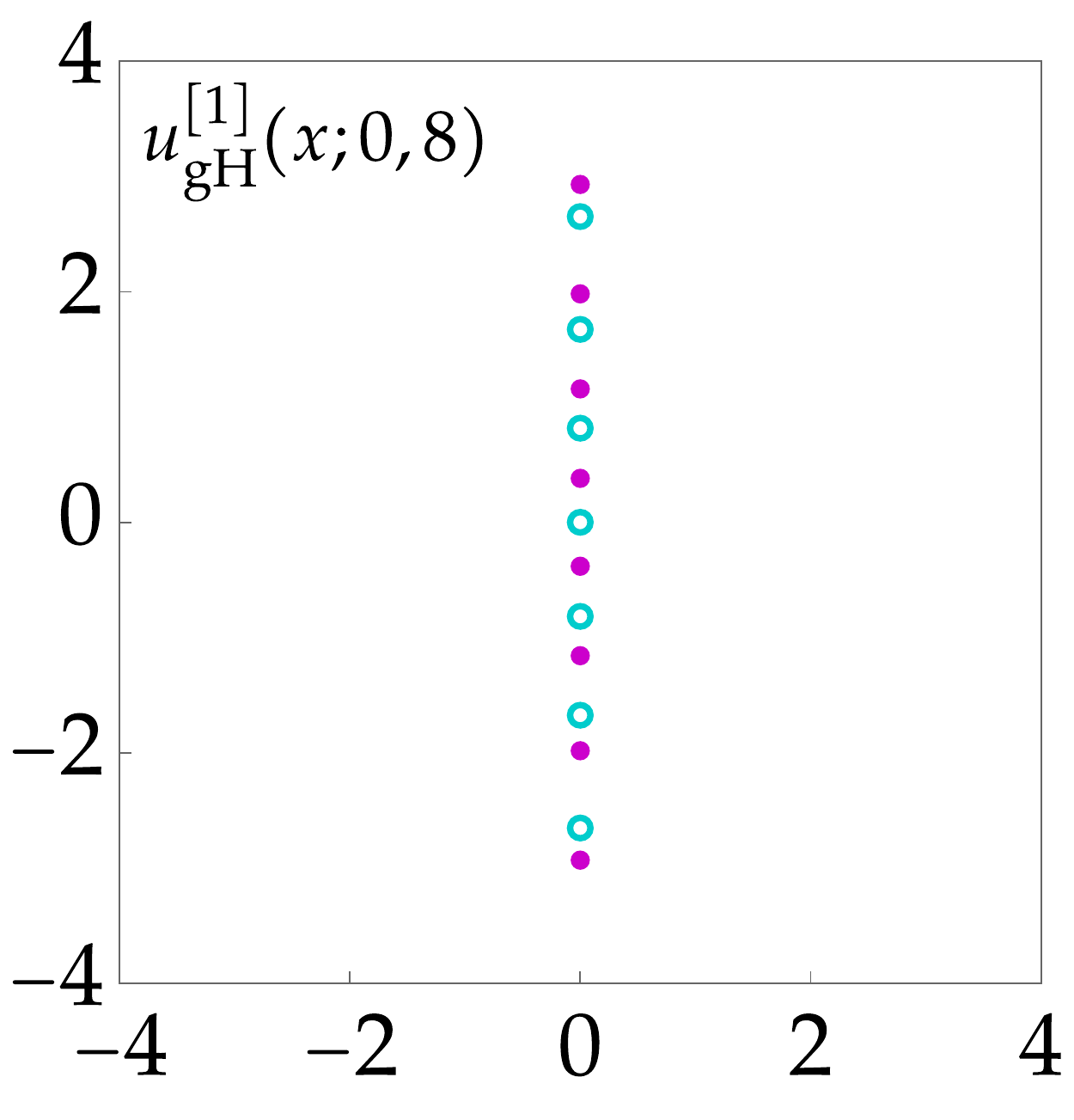} \vspace{3.75in} \\
\includegraphics[height=1.2in]{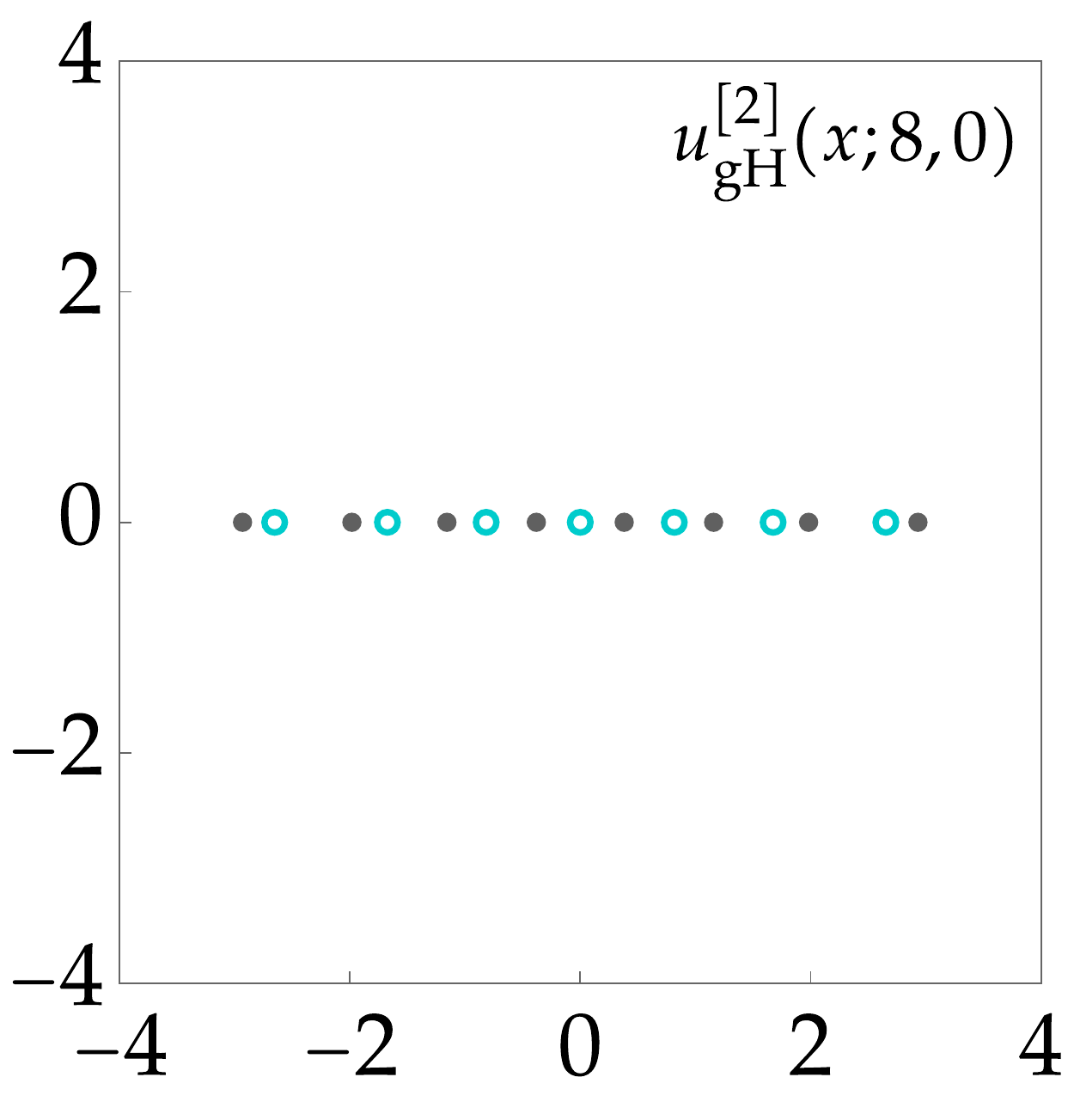}
\end{tabular}
\hspace{-.18in}
\begin{tabular}{c}
\includegraphics[height=1.2in]{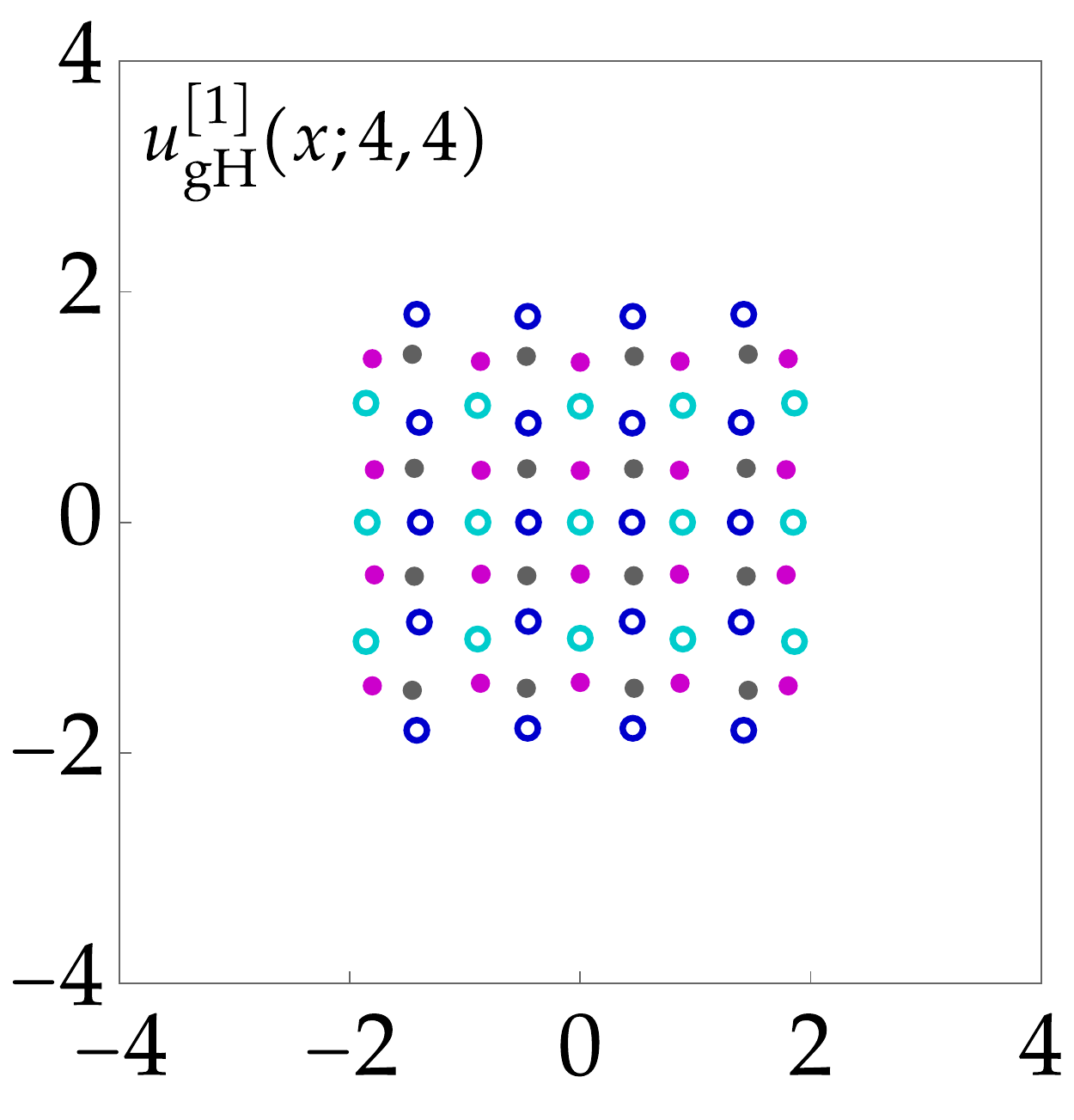}
\includegraphics[height=1.2in]{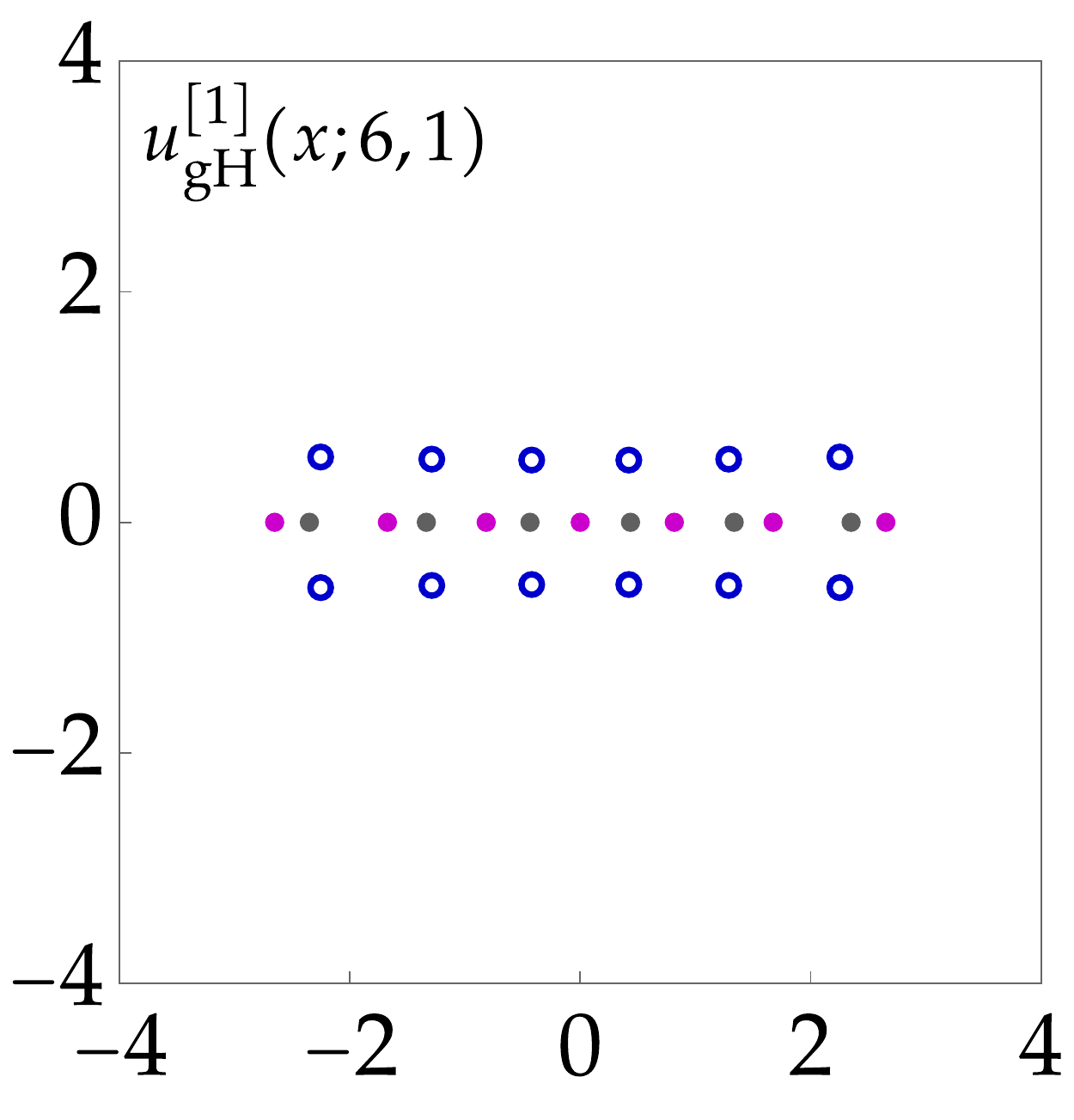}
\hspace{1.2in} \\
\includegraphics[width=3.7in]{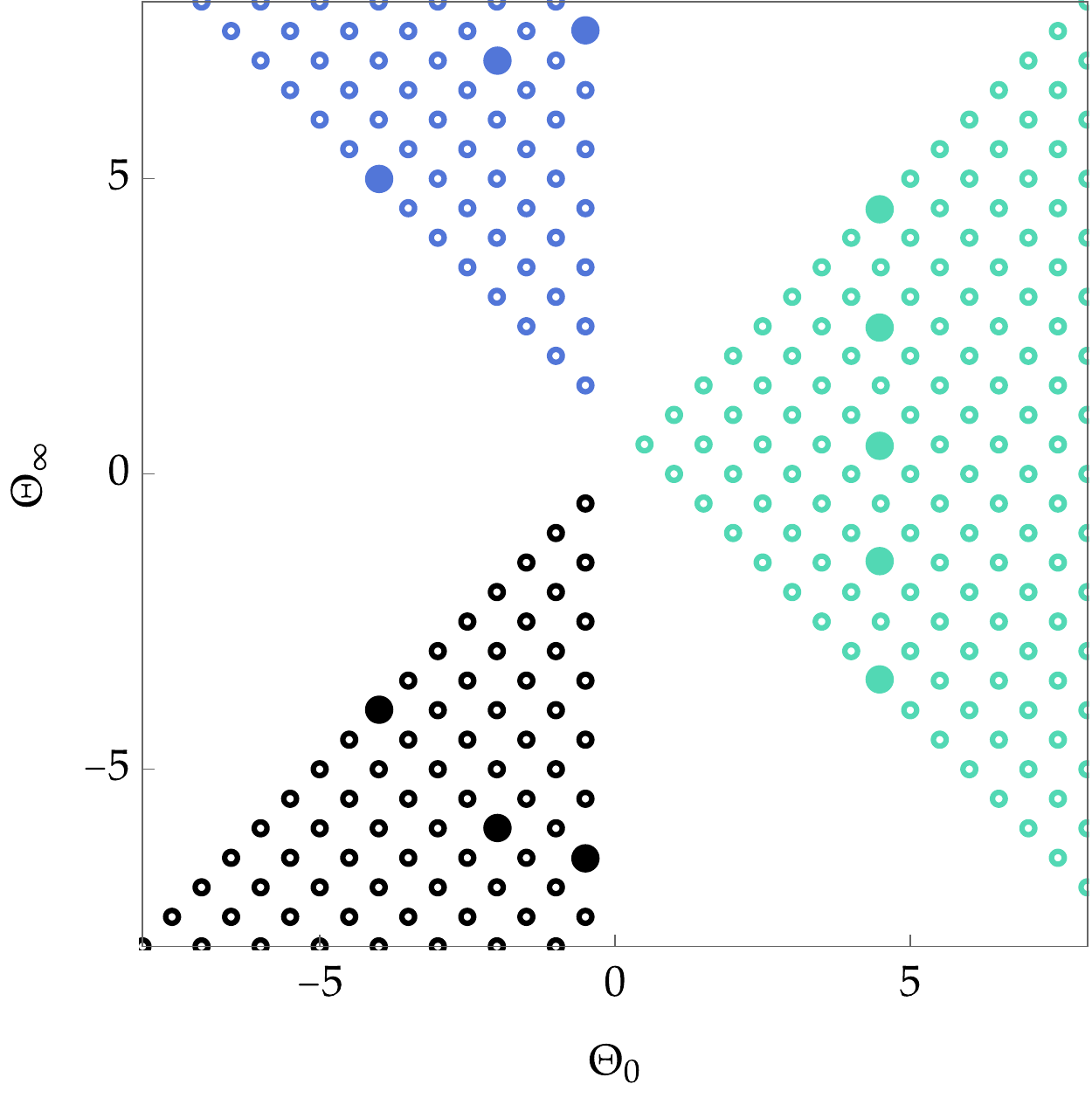}\\
\includegraphics[height=1.2in]{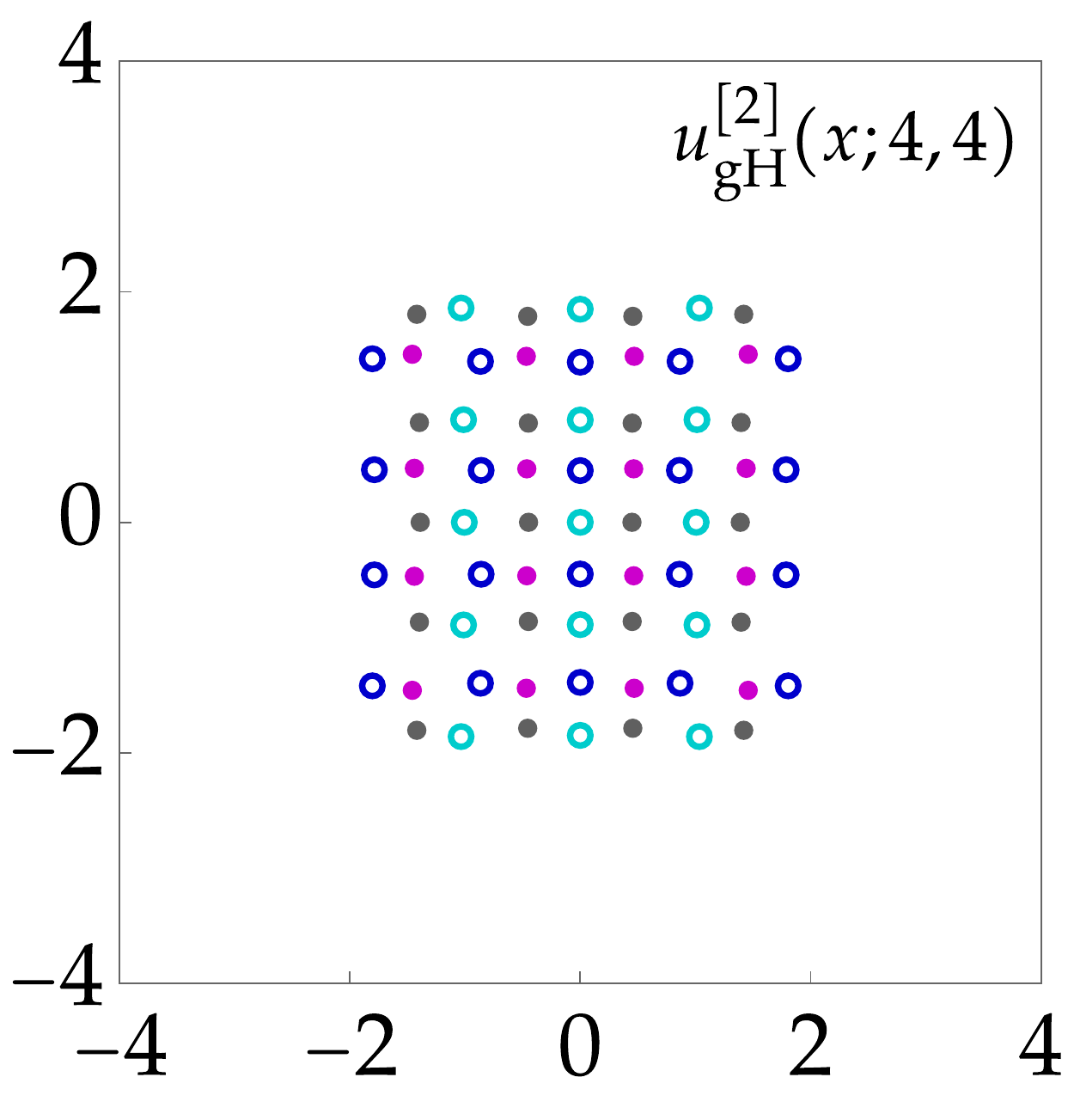}
\includegraphics[height=1.2in]{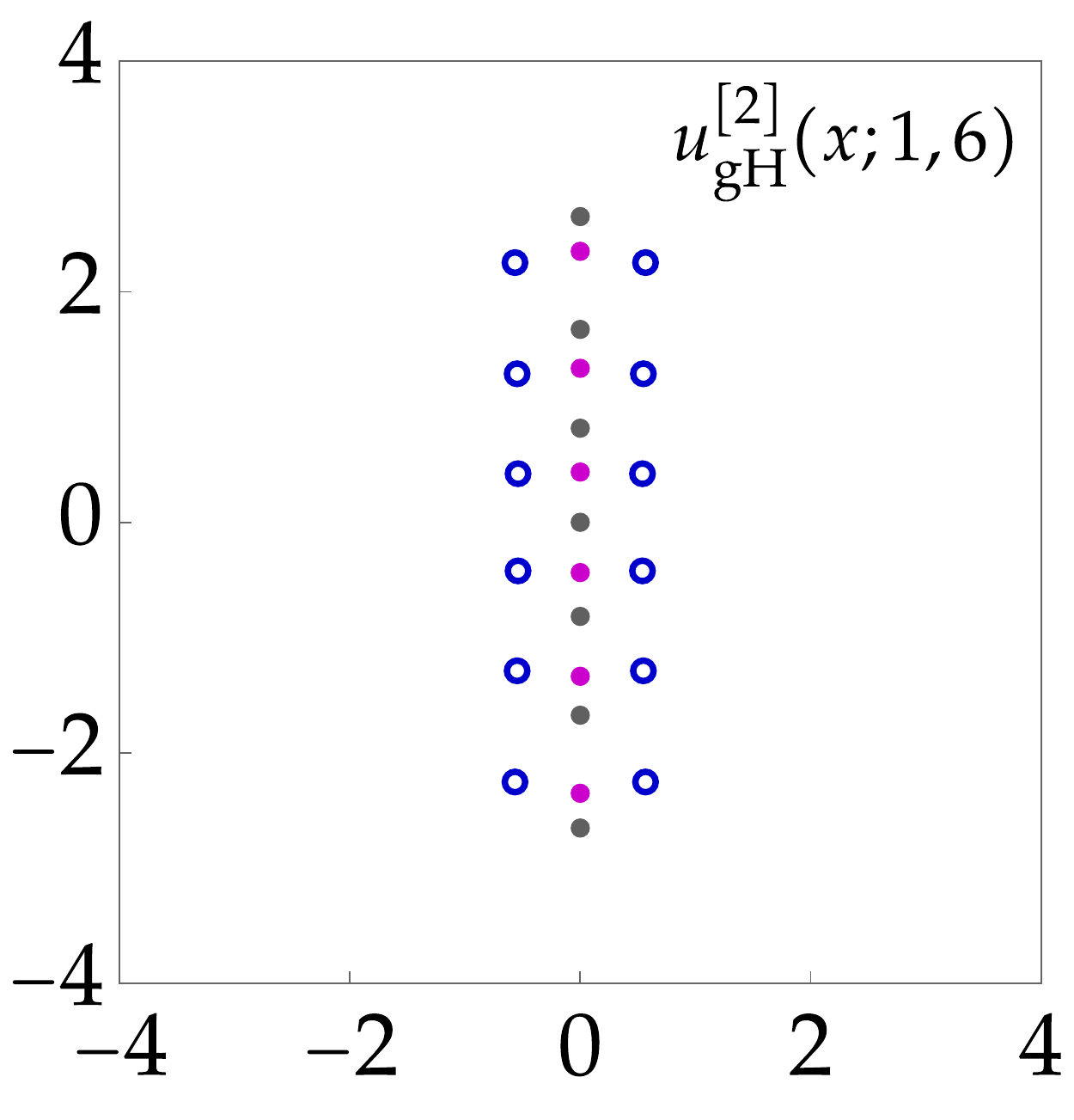}
\hspace{1.2in} \\
\end{tabular}
\hspace{-.18in}
\begin{tabular}{c}
\includegraphics[height=1.2in]{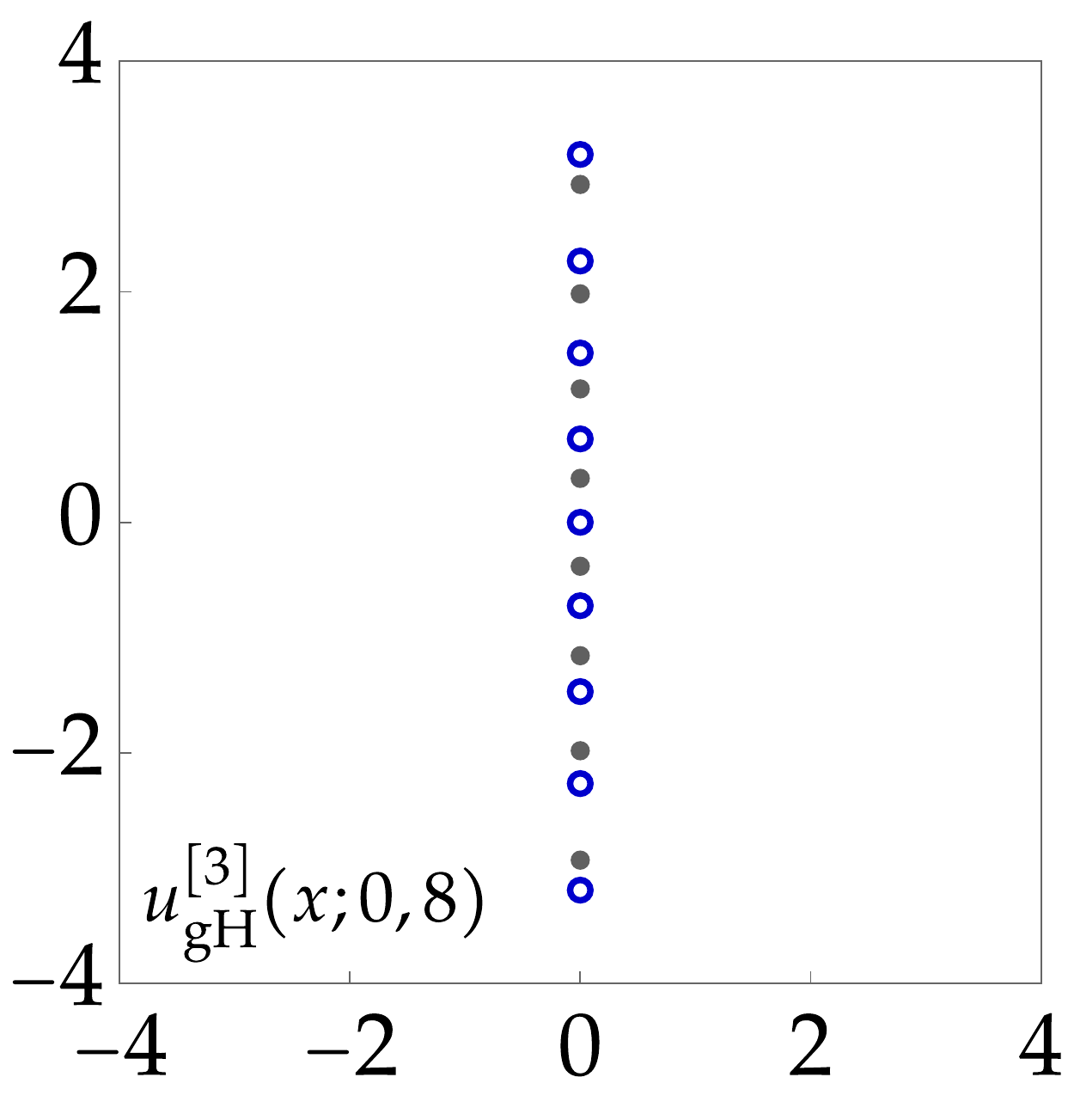}\\
\includegraphics[height=1.2in]{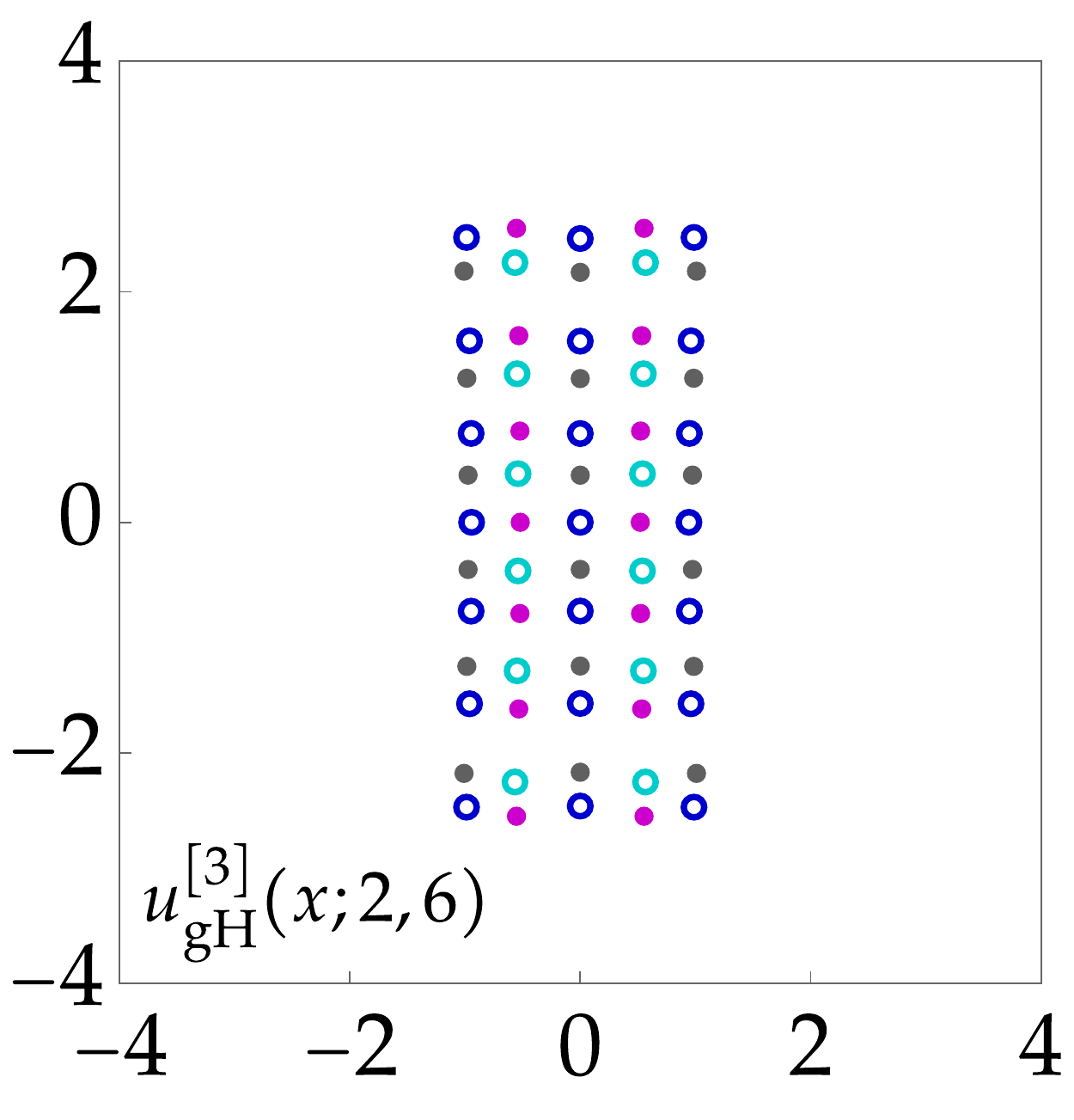} \\
\includegraphics[height=1.2in]{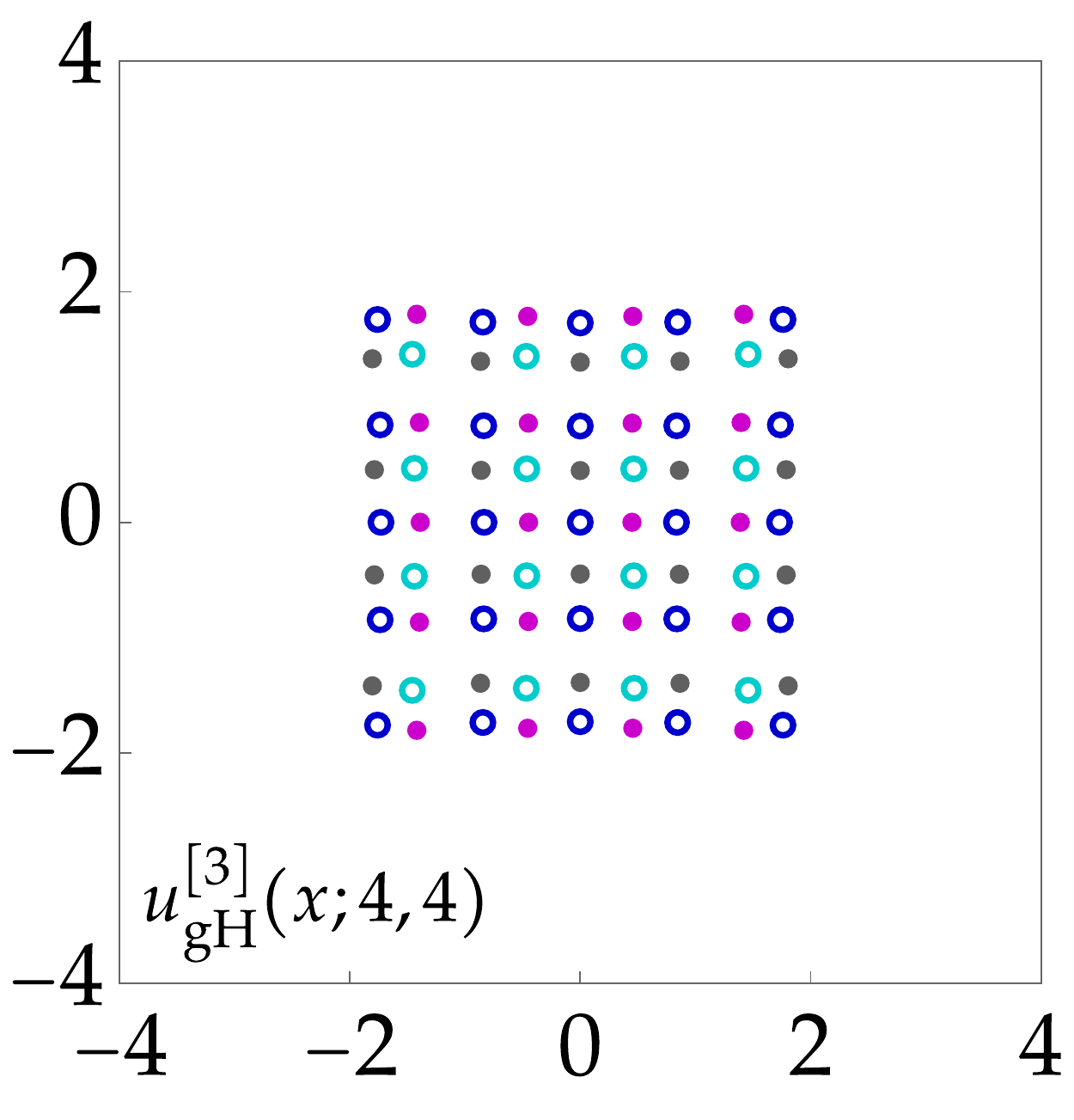}\\
\includegraphics[height=1.2in]{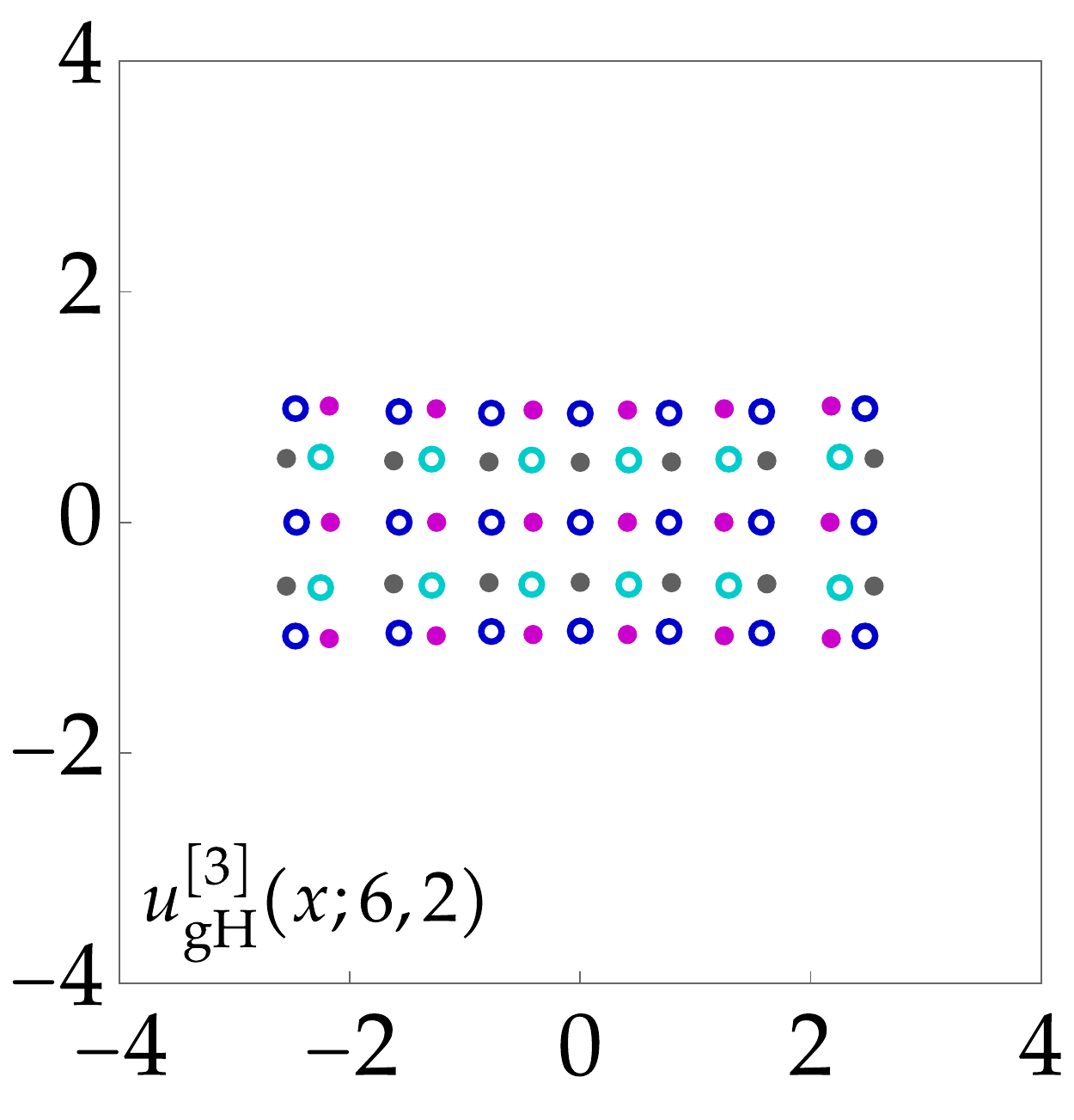}\\
\includegraphics[height=1.2in]{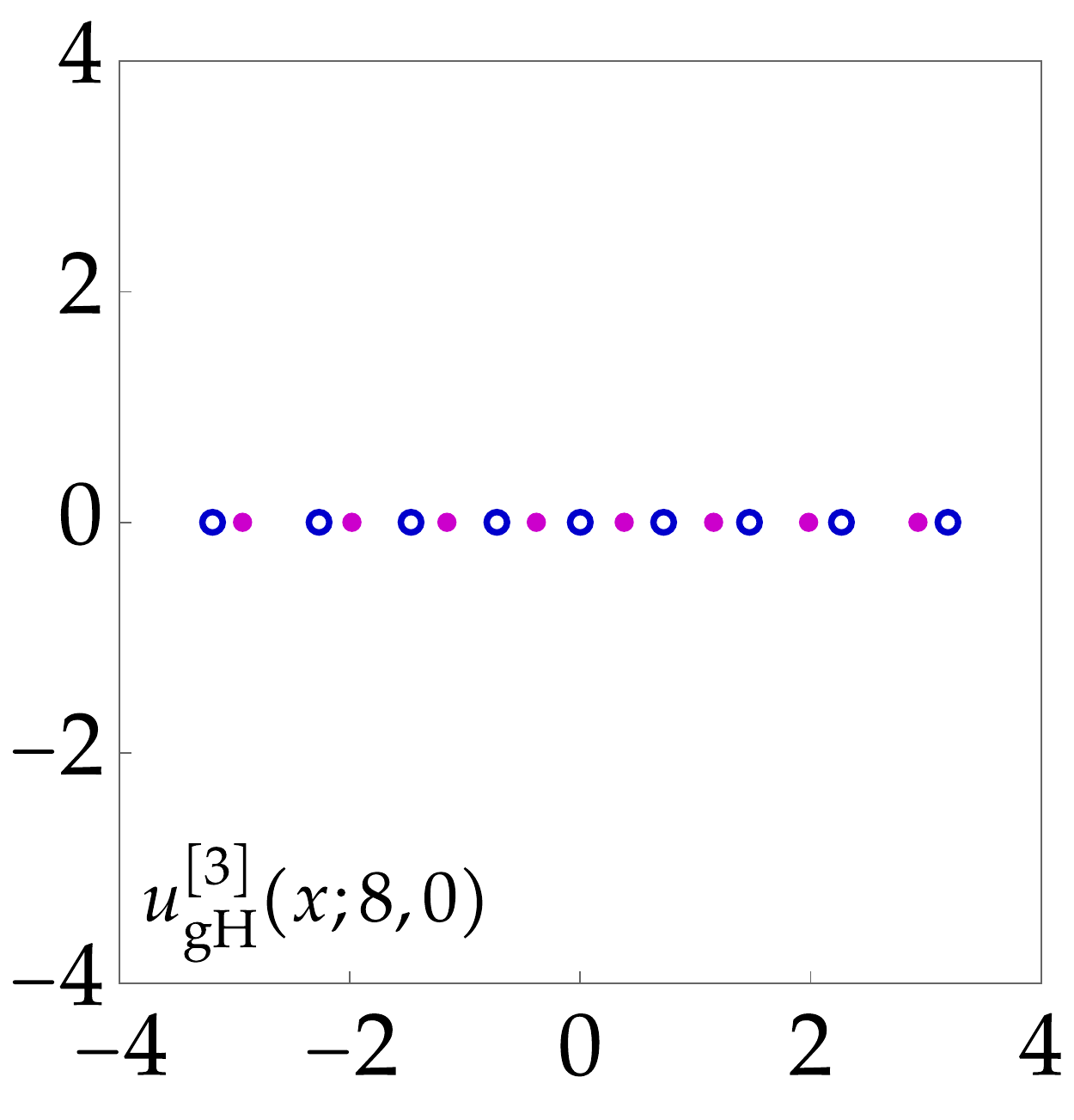}
\end{tabular}
\caption{Representative plots of poles (dots; magenta for residue $+1$ and gray for residue $-1$) and zeros (circles; cyan for positive derivative and blue for negative derivative) of the three types of rational solutions in the gH family.  The large dots in the central plot show the location of the corresponding parameters in the $(\Theta_0,\Theta_\infty)$-plane.}
\label{fig:theta-map-gH}
\end{figure}
The analogous information for the gO case is shown in Figure~\ref{fig:theta-map-gO}.
\begin{figure}[h]
\hspace{-0.2in}
\begin{tabular}{c}
\includegraphics[height=1.2in]{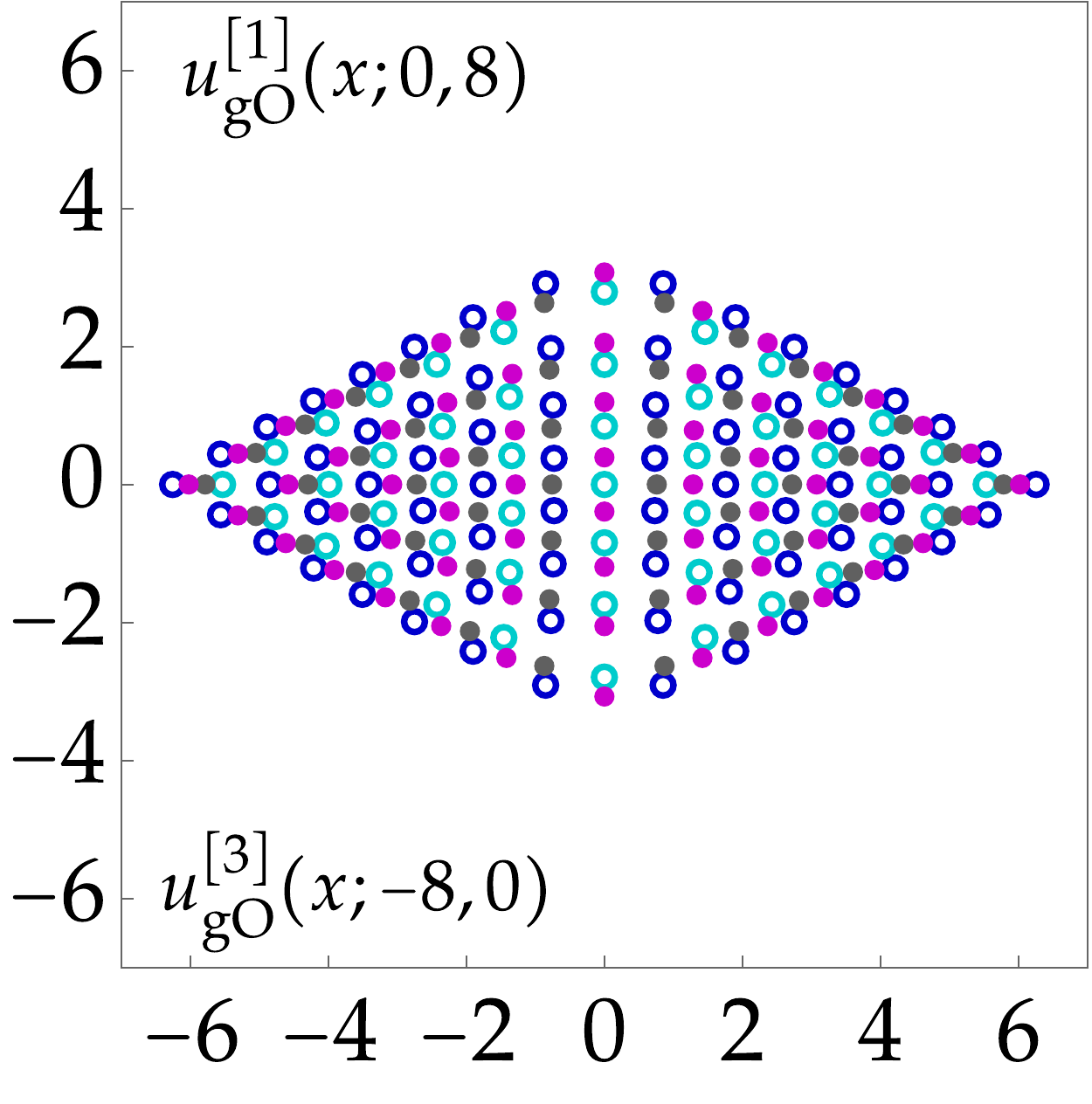}\\
\includegraphics[height=1.2in]{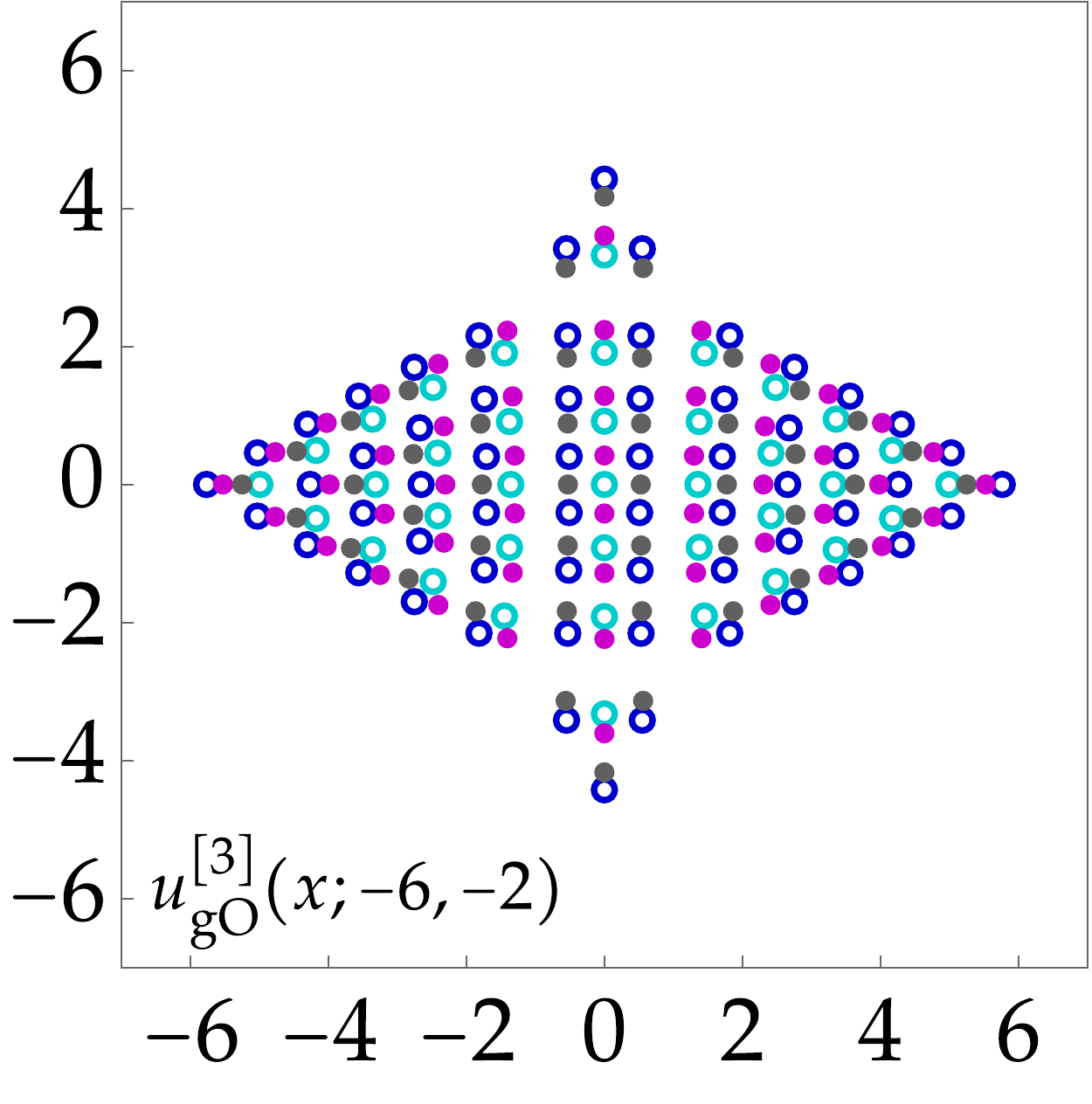}\\
\includegraphics[height=1.2in]{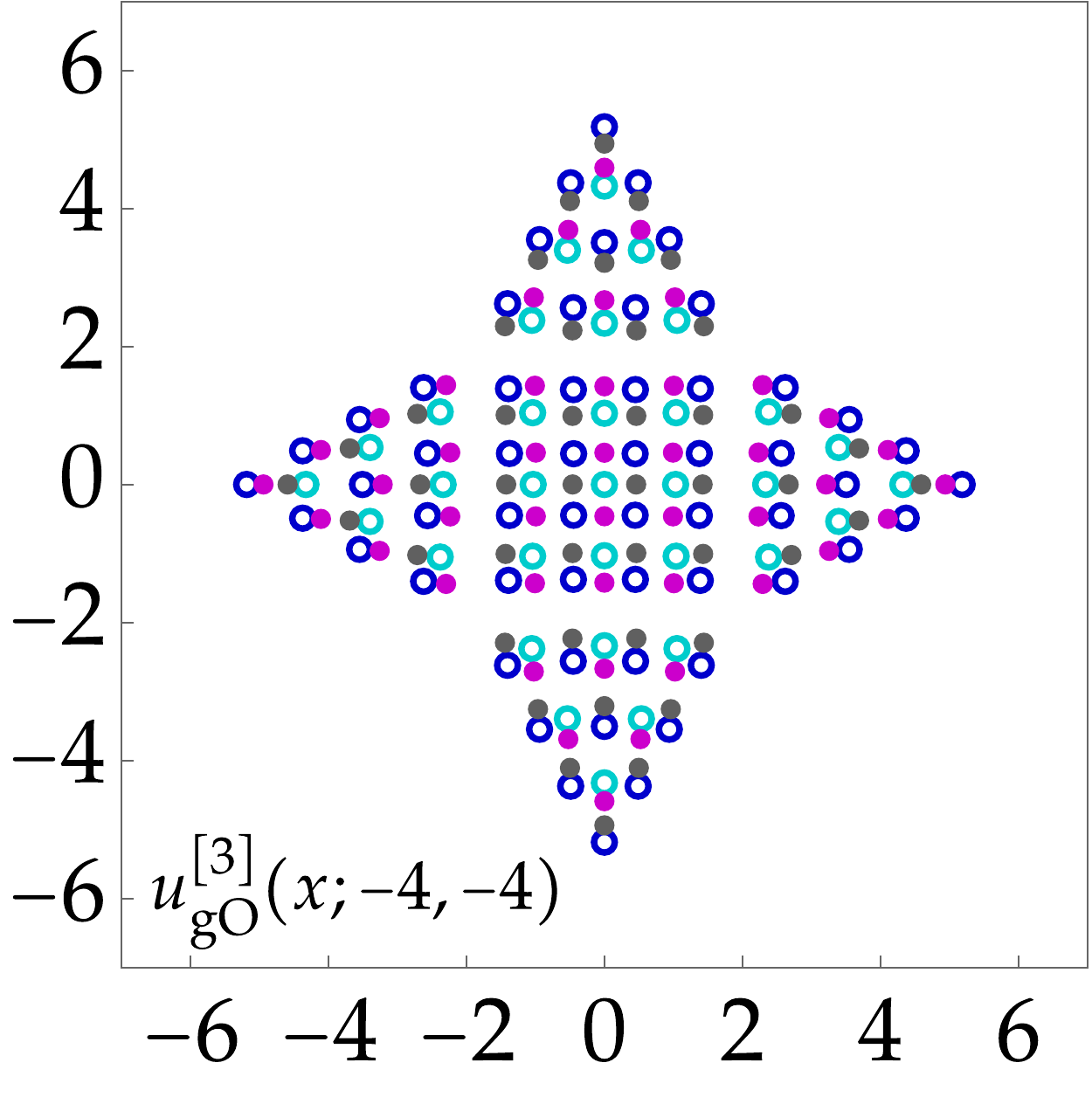}\\
\includegraphics[height=1.2in]{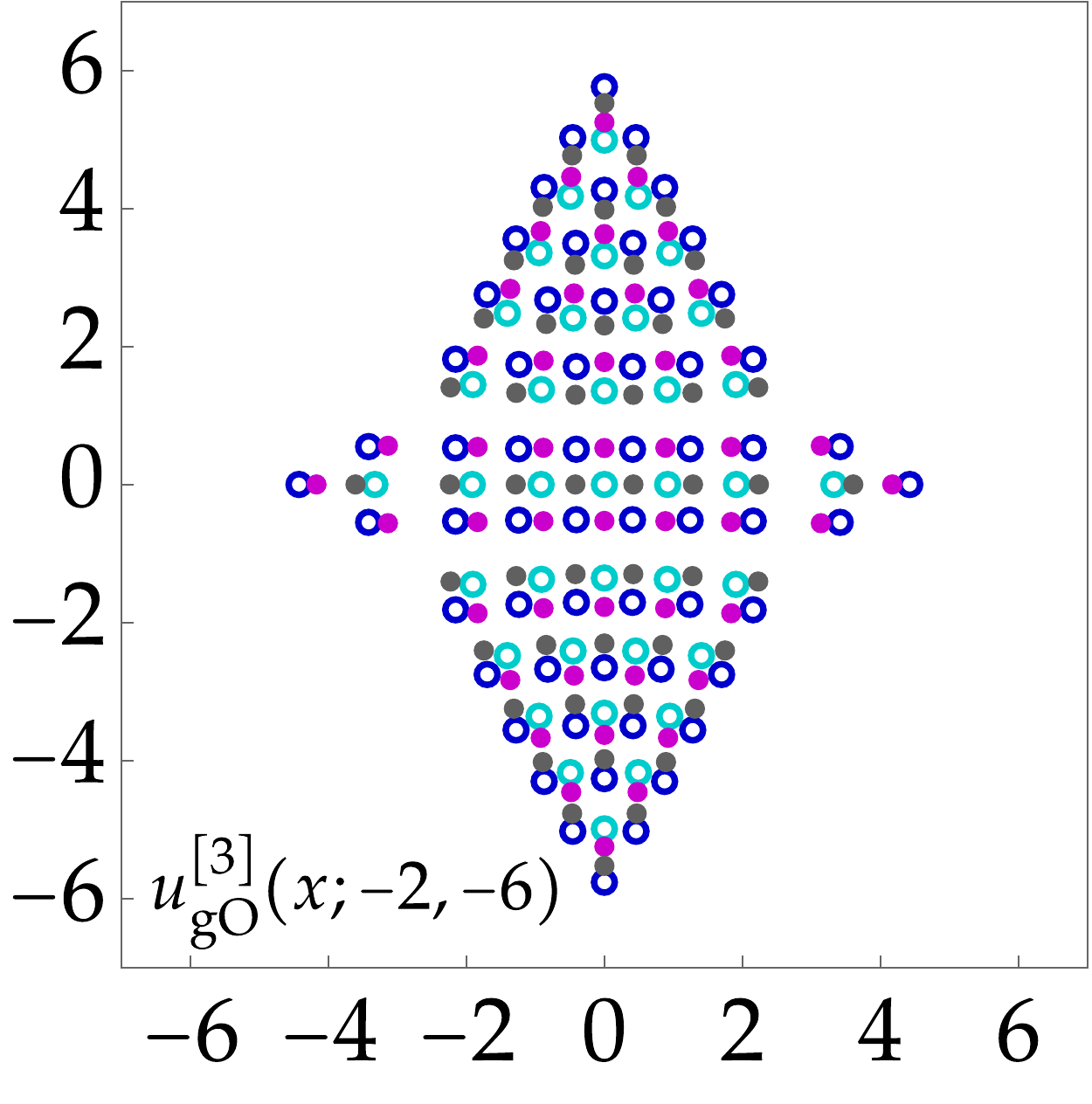}\\
\includegraphics[height=1.2in]{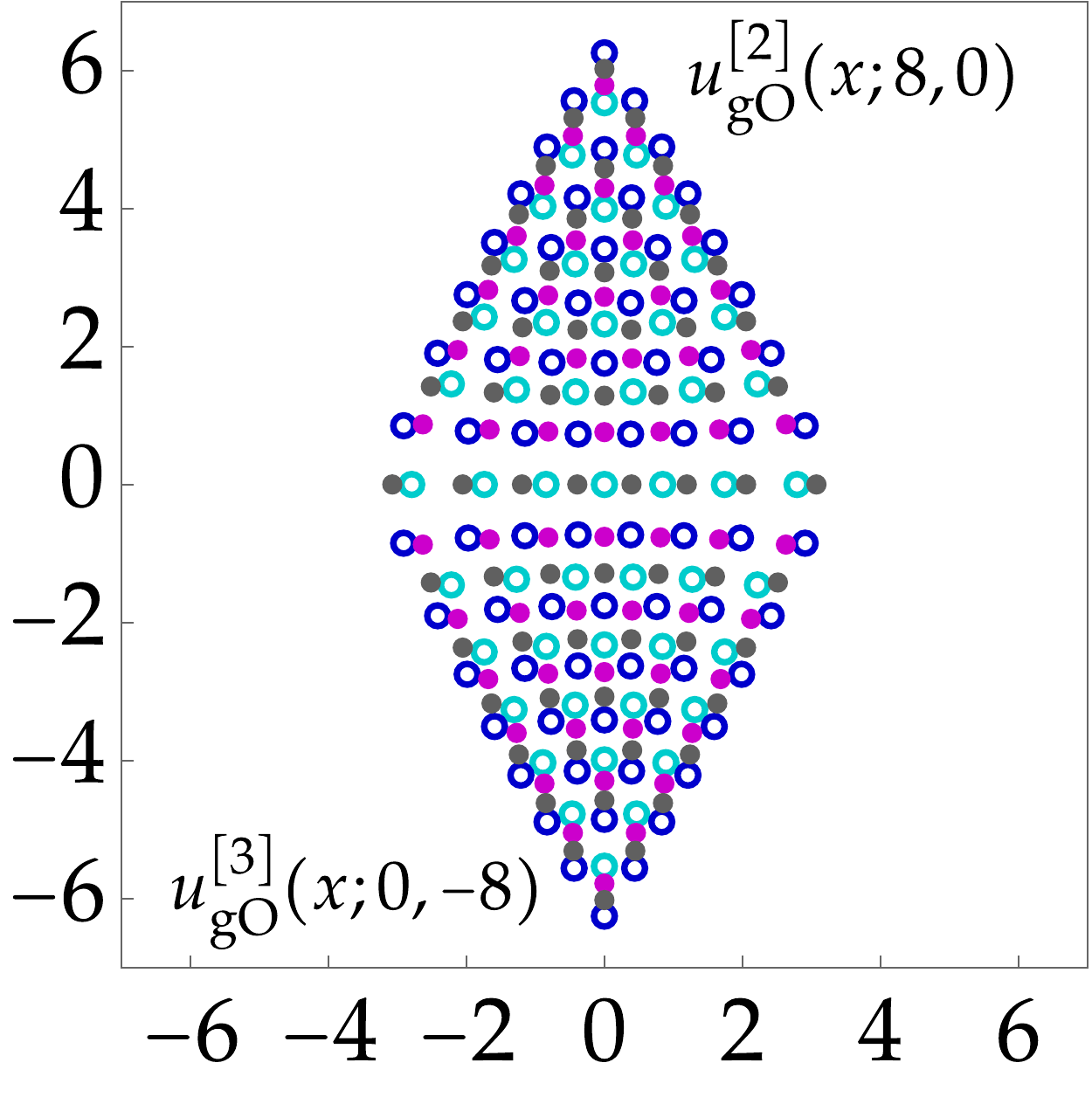}
\end{tabular}
\hspace{-.18in}
\begin{tabular}{c}
\includegraphics[height=1.2in]{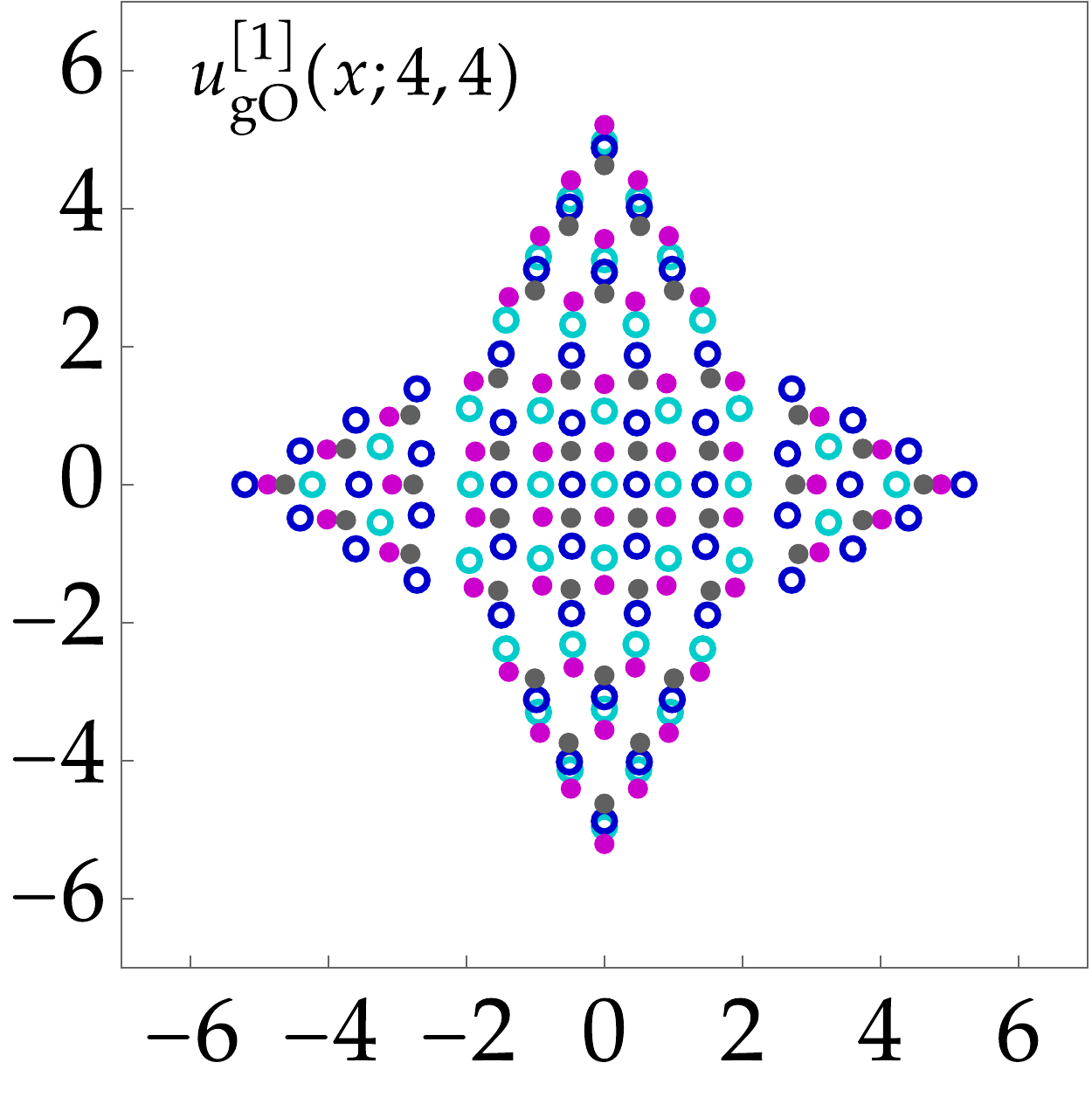}
\includegraphics[height=1.2in]{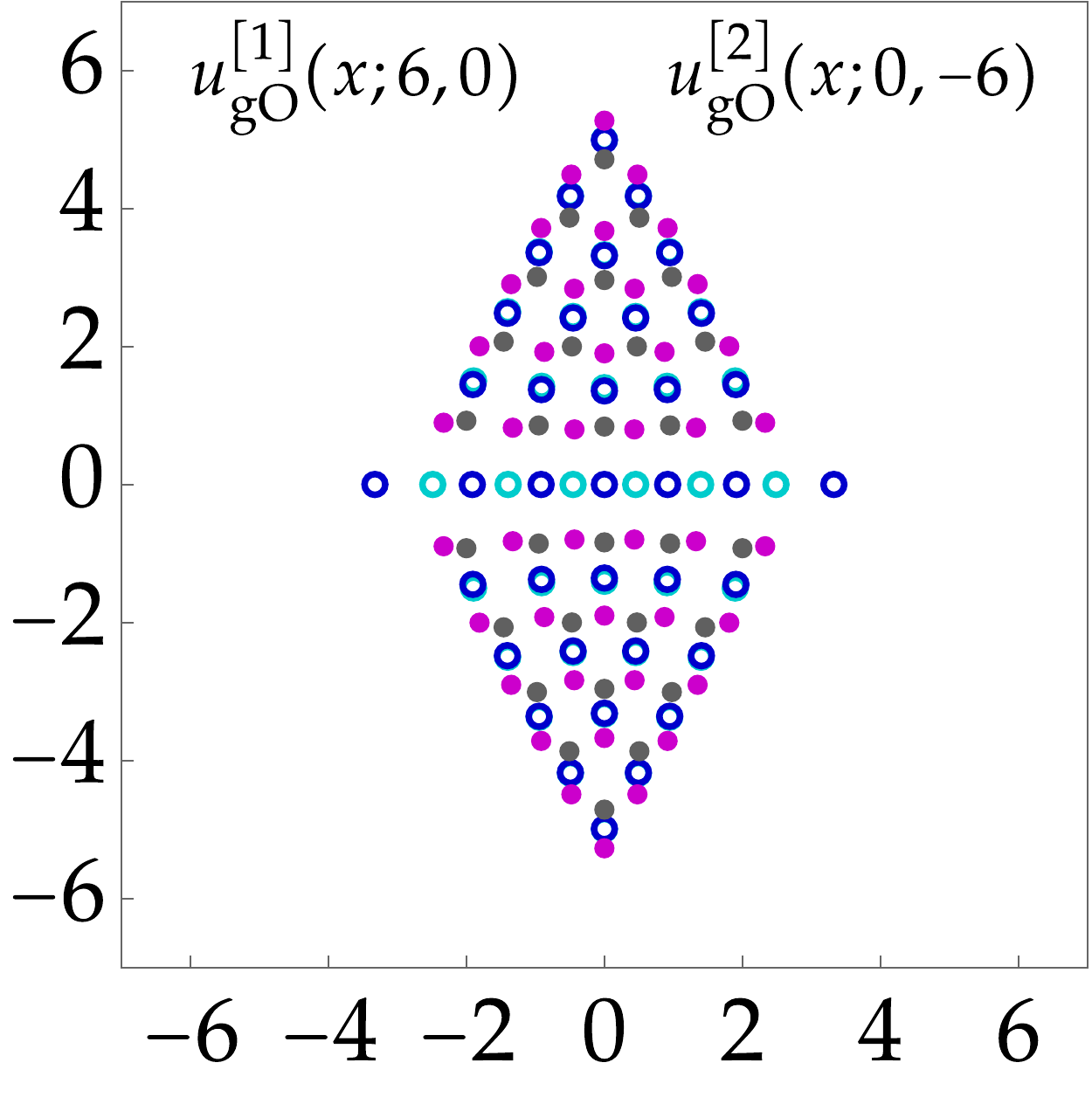}
\includegraphics[height=1.2in]{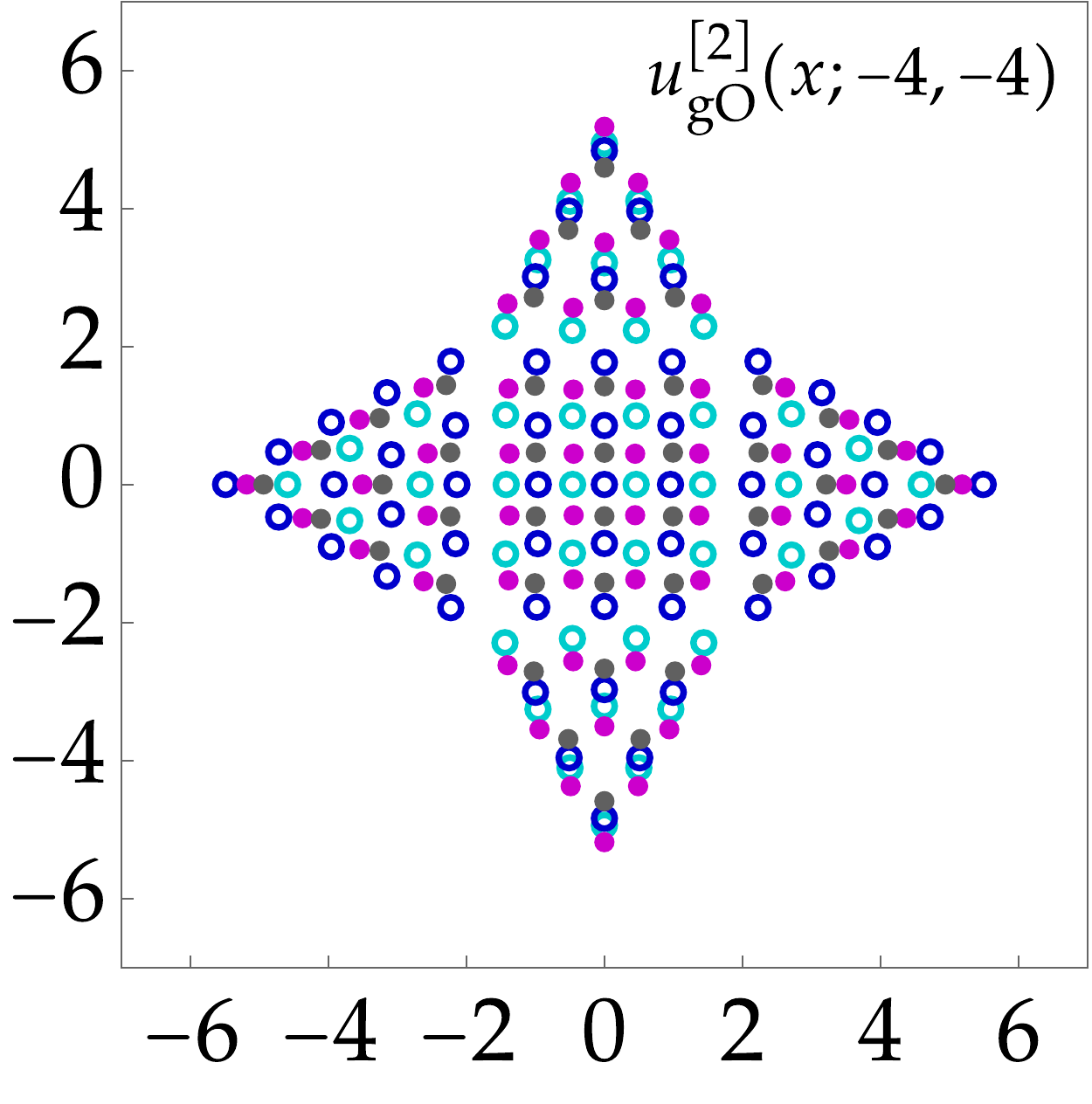}\\
\includegraphics[width=3.7in]{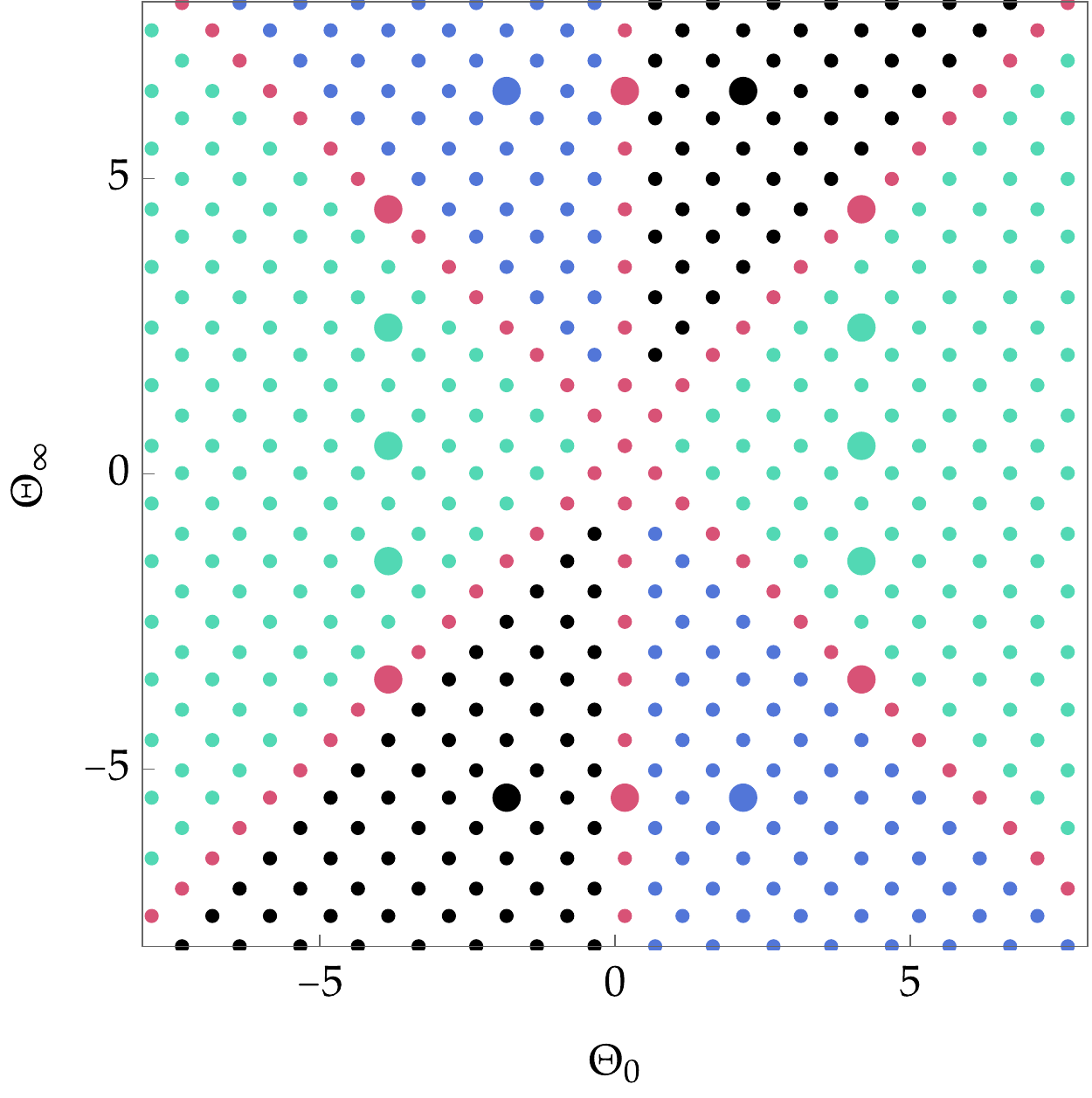}\\
\includegraphics[height=1.2in]{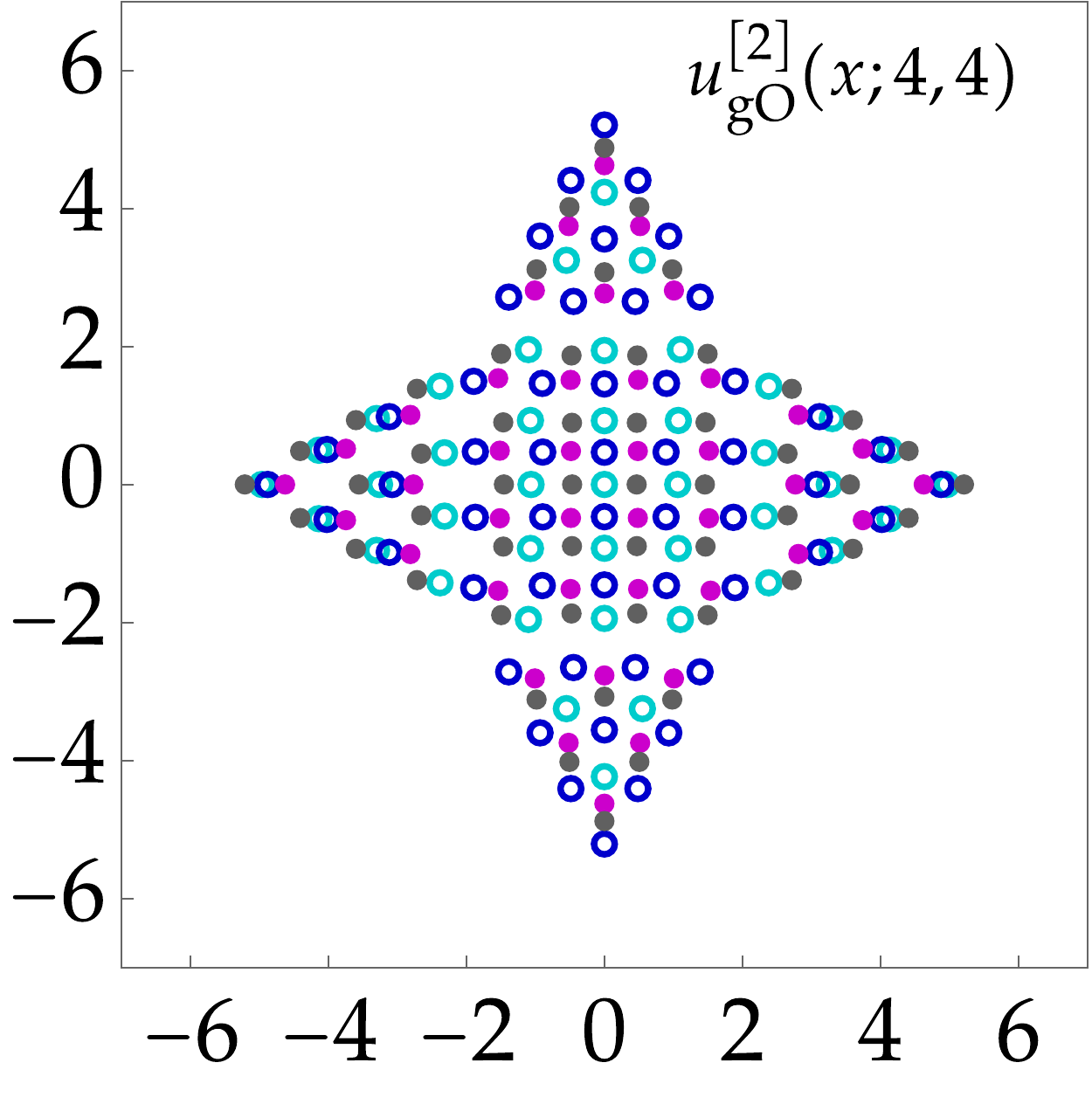}
\includegraphics[height=1.2in]{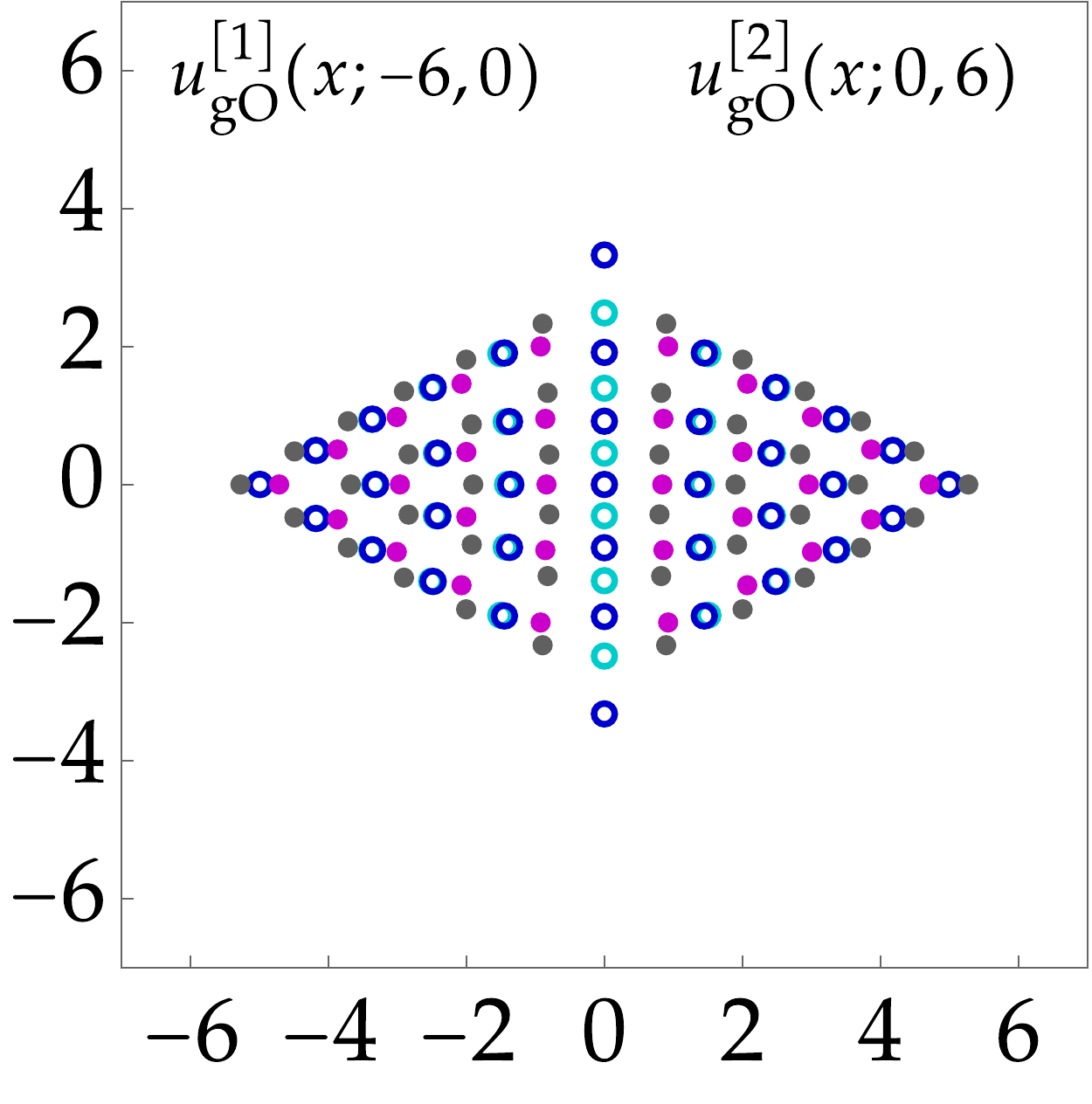}
\includegraphics[height=1.2in]{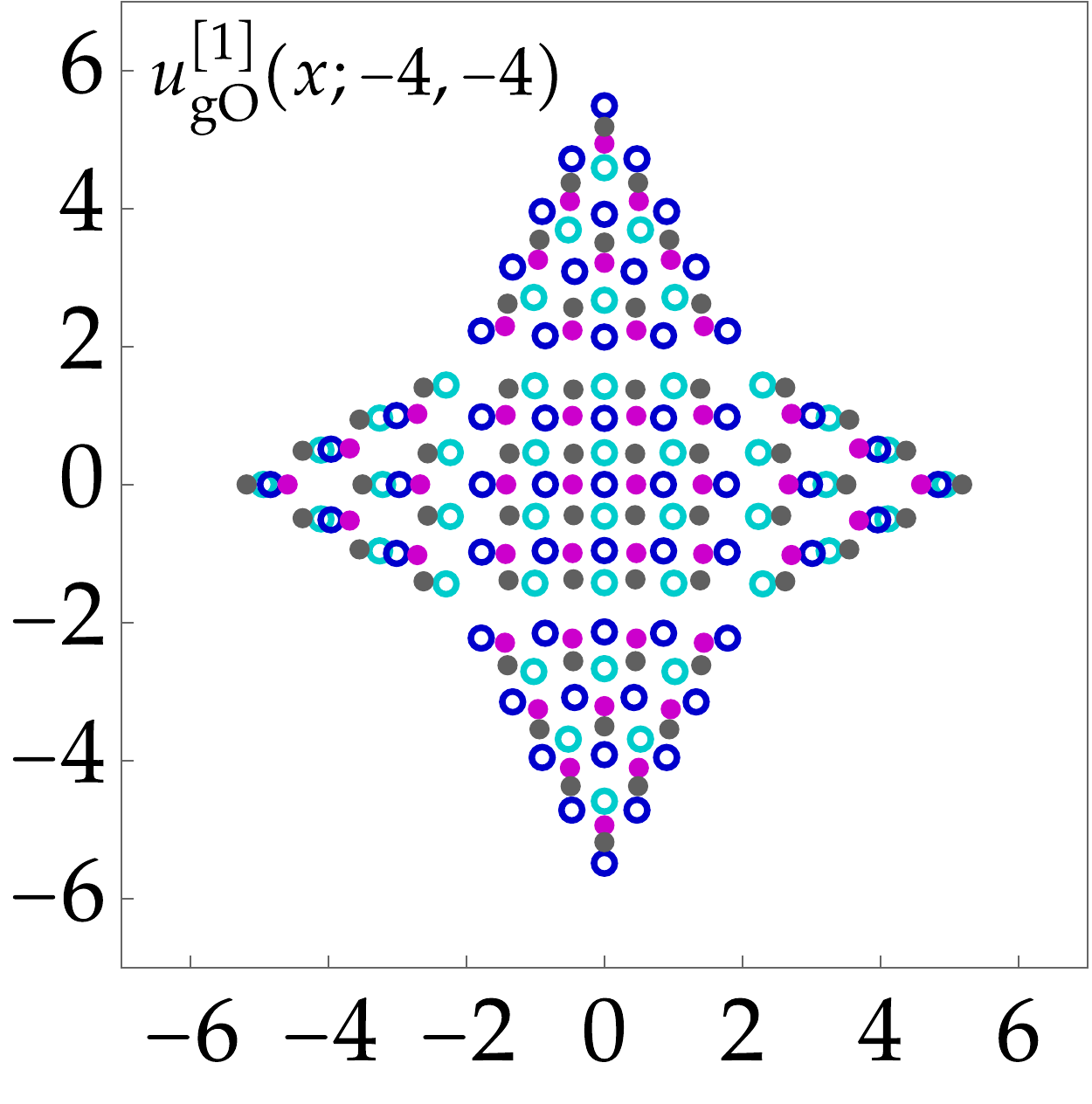}
\end{tabular}
\hspace{-.18in}
\begin{tabular}{c}
\includegraphics[height=1.2in]{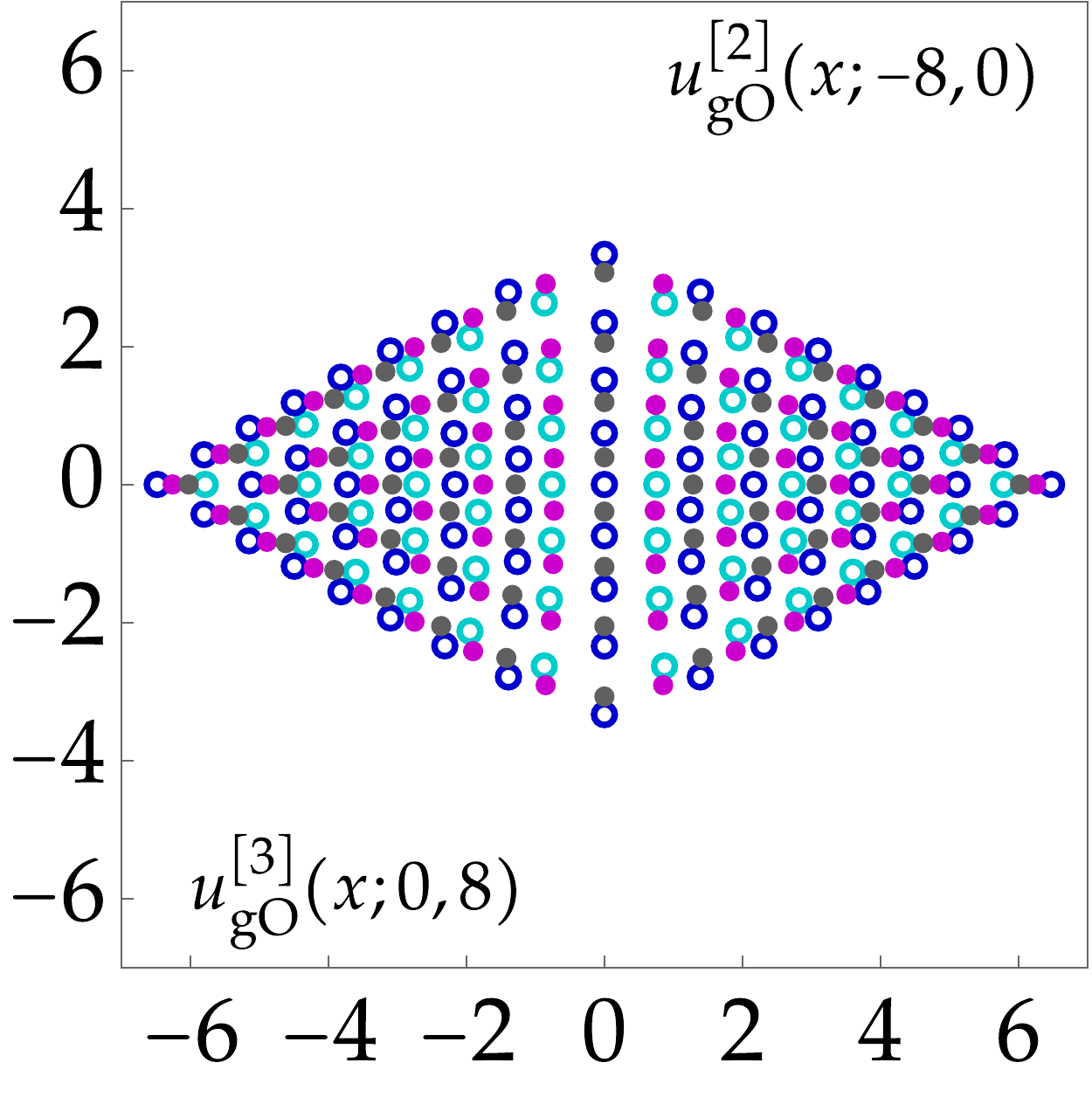}\\
\includegraphics[height=1.2in]{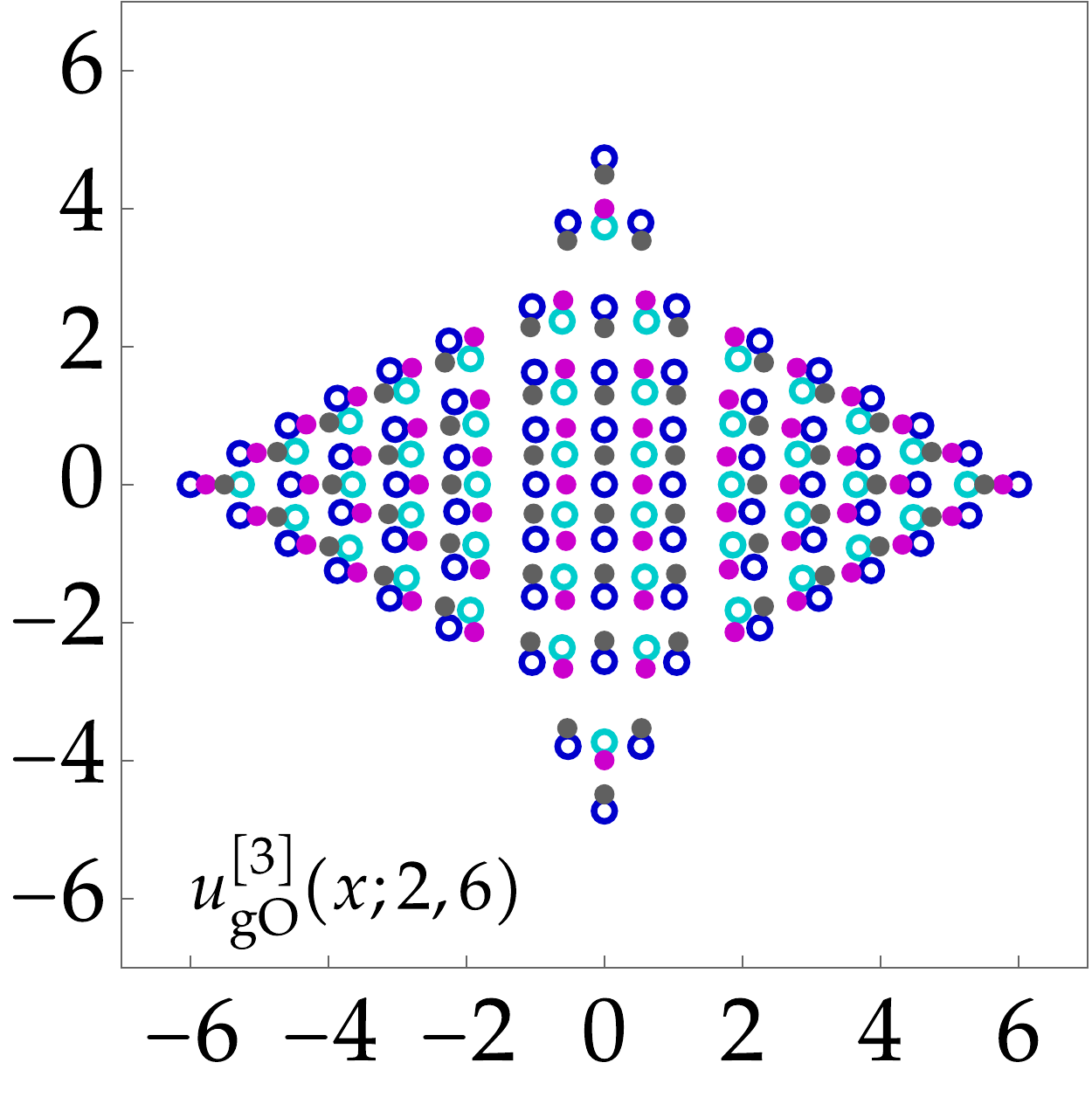}\\
\includegraphics[height=1.2in]{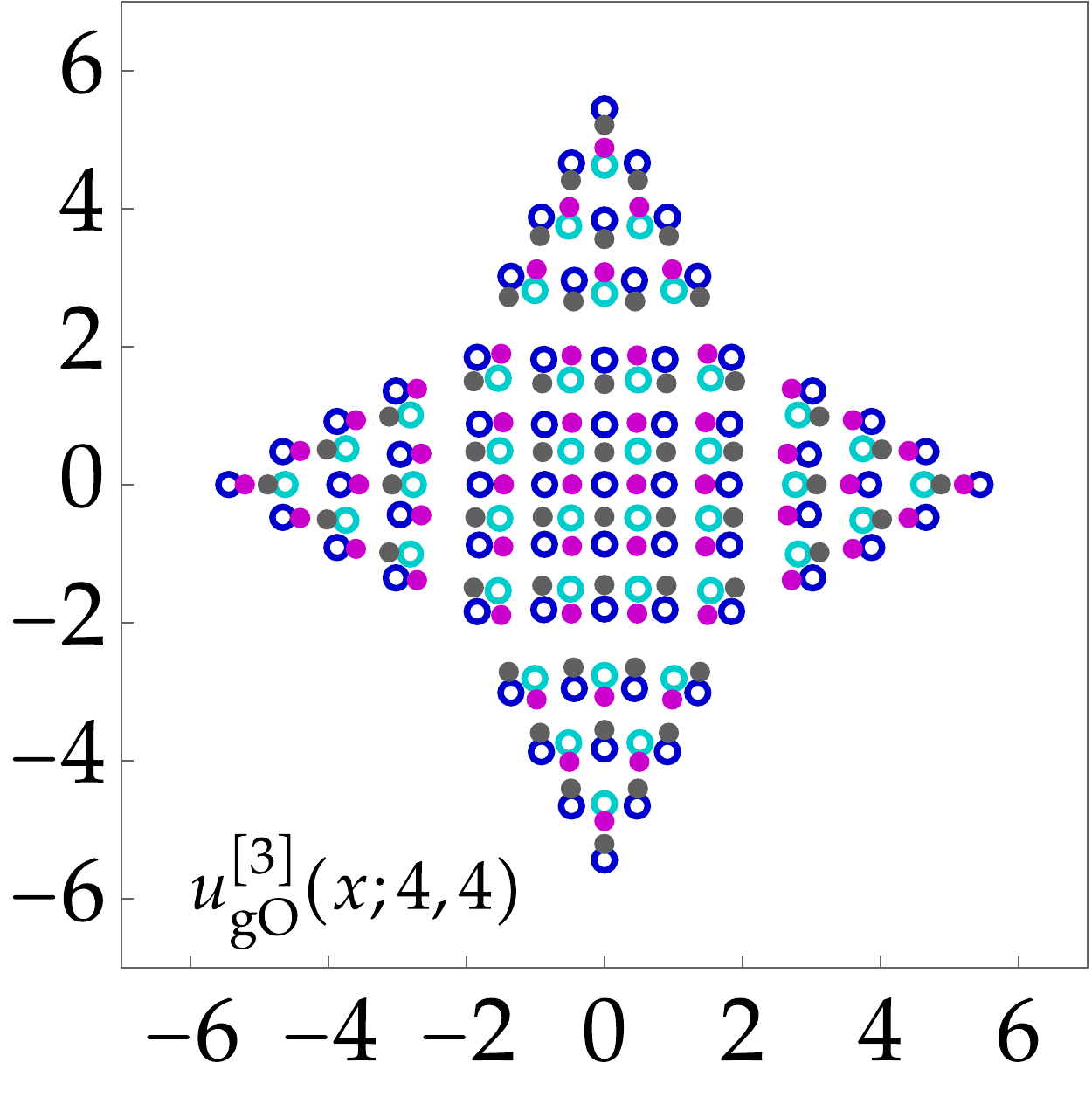}\\
\includegraphics[height=1.2in]{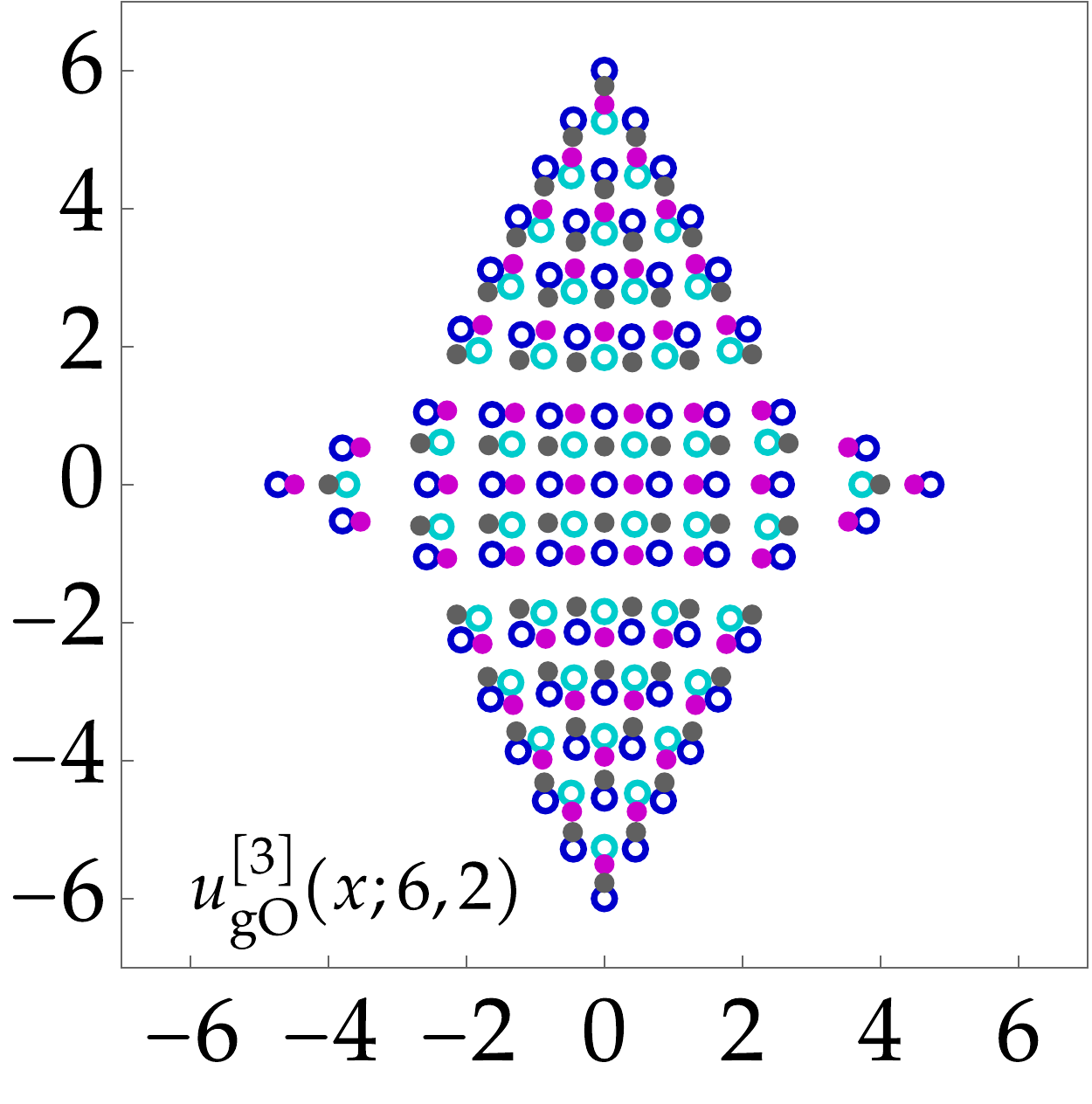}\\
\includegraphics[height=1.2in]{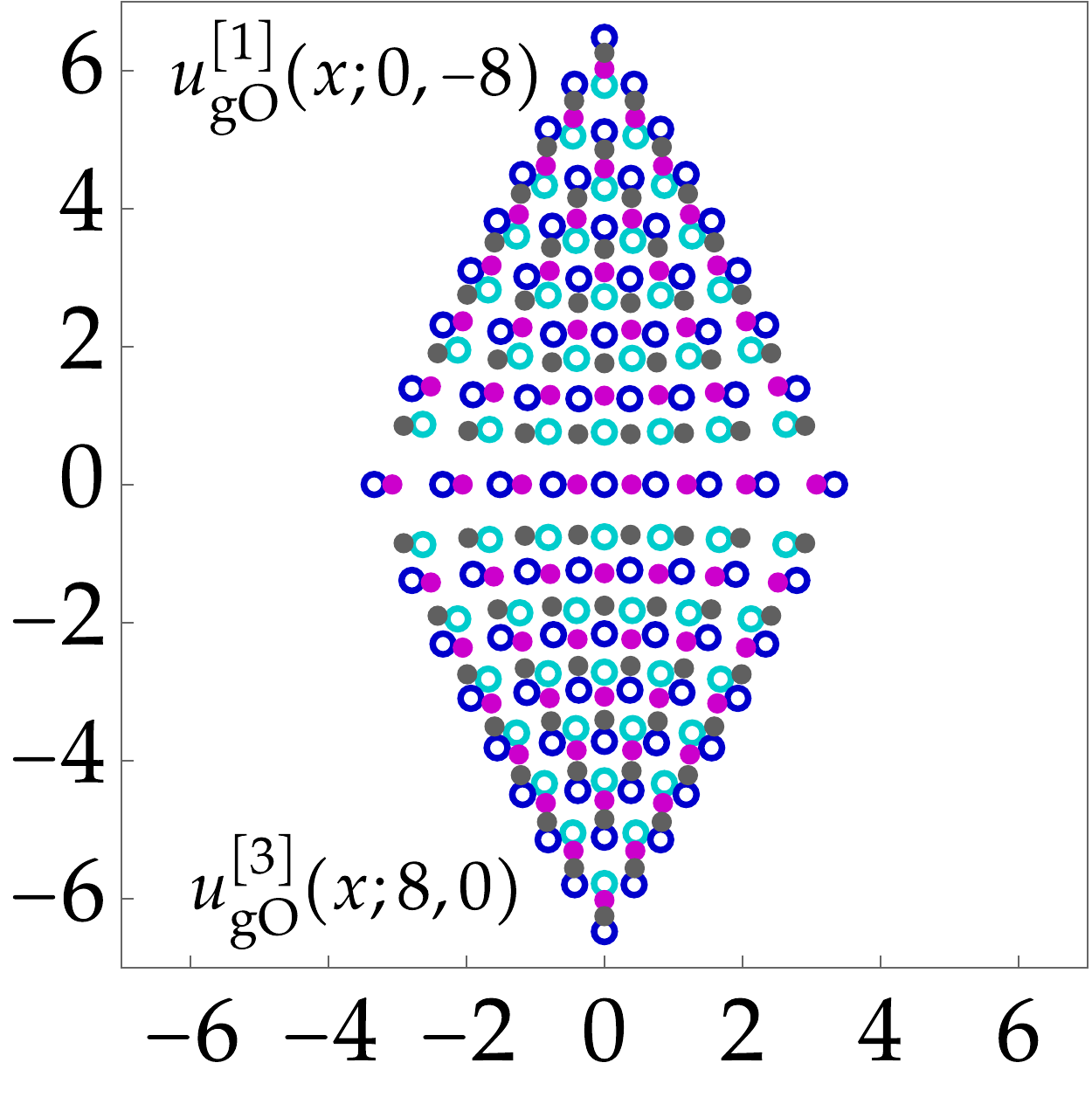}
\end{tabular}
\caption{As in Figure~\ref{fig:theta-map-gH} but for the gO family of rational solutions.}
\label{fig:theta-map-gO}
\end{figure}

\subsubsection{B\"acklund transformations and symmetries}
\label{sec:intro-Baecklund}
Each of the four parameter sets $\Lambda_\mathrm{gH}^{[1]-}$, $\Lambda_\mathrm{gH}^{[2]-}$, $\Lambda_\mathrm{gH}^{[3]+}$, and $\Lambda_\mathrm{gO}$, along with its corresponding rational solutions of \eqref{p4}, can be generated from a ``seed'' triple $(\Theta_0,\Theta_\infty,u(x))$ by iteratively applying certain B\"acklund transformations in order to increment or decrement the parameters $(\Theta_0,\Theta_\infty)$ by a basis (over $\mathbb{Z}$) of lattice vectors.  The basis can be chosen to correspond to integer increments of $m$ and $n$.  Such B\"acklund transformations are \emph{isomonodromic} and they lift to Schlesinger transformations of corresponding simultaneous solutions of the Lax pair for Painlev\'e-IV; see Section~\ref{sec:Schlesinger} for details.  For each type of rational solutions in the gH family, a basic isomonodromic B\"acklund transformation becomes indeterminate (i.e., producing an identically vanishing denominator) for the transformed rational solution if it is applied at a lattice point on the boundary of the parameter sector and would yield transformed parameters one step outside the sector.  For example, at the points $(\Theta_0,\Theta_\infty)=(\tfrac{1}{2}+\tfrac{1}{2}n,\tfrac{1}{2}+\tfrac{1}{2}n)$, $n=0,1,2,\dots$, that make up part of the boundary of $\Lambda_\mathrm{gH}^{[3]+}$, the isomonodromic B\"acklund transformation $(\Theta_0,\Theta_\infty,u(x))\mapsto (\Theta_0+\tfrac{1}{2},\Theta_\infty-\tfrac{1}{2},u_\searrow(x))$ (see \eqref{eq:u-sw}) is indeterminate as it leads to negative values of $m$.  The same transformation is however valid at all other points of $\Lambda_{\mathrm{gH}}^{[3]+}$.  The reason this occurs is that the three sets of monodromy data consisting of Stokes and connection matrices for the three types of gH rational solutions of \eqref{p4} are all different.  This phenomenon does not occur for the gO family, each point of which yields the same monodromy data (but different again from that of all three gH types).

Therefore, if it is desired to generate one type of rational solution of \eqref{p4} in the gH family from  another type of solution in the same family, one must apply a B\"acklund transformation that is \emph{not} isomonodromic.  We will find useful two such transformations, both of which can correspond to large leaps in the $(\Theta_0,\Theta_\infty)$-plane rather than incremental steps, and hence might be thought of as global symmetries of the Painlev\'e-IV equation.  Firstly, there is an elementary symmetry $\mathcal{S}_{\updownarrow}$ of \eqref{p4}
\eq
\mathcal{S}_{\updownarrow}(\Theta_0,\Theta_\infty,u(x))=(\Theta_{0,\updownarrow},\Theta_{\infty,\updownarrow},u_\updownarrow(x)):=(\Theta_0,1-\Theta_\infty,\ii u(-\ii x))
\label{eq:rotation-symmetry}
\endeq
due to Boiti and Pempinelli \cite{BoitiP80}.  The action of $\mathcal{S}_\updownarrow$ in the parameter space is merely a reflection through the horizontal line $\Theta_\infty=\tfrac{1}{2}$, which obviously preserves both lattices $\Lambda_\mathrm{gH}$ and $\Lambda_\mathrm{gO}$.  In particular, note that for either family $\text{F}=\text{gH}$ or $\text{F}=\text{gO}$,
\eq
\begin{split}
(\Theta_0,\Theta_\infty)=(\Theta^{[1]}_{0,\mathrm{F}}(m,n),\Theta^{[1]}_{\infty,\mathrm{F}}(m,n))&\quad\implies\quad
(\Theta_{0,\updownarrow},\Theta_{\infty,\updownarrow})=(\Theta^{[2]}_{0,\mathrm{F}}(n,m),\Theta^{[2]}_{\infty,\mathrm{F}}(n,m))\\
(\Theta_0,\Theta_\infty)=(\Theta^{[3]}_{0,\mathrm{F}}(m,n),\Theta^{[3]}_{\infty,\mathrm{F}}(m,n))&\quad\implies\quad (\Theta_{0,\updownarrow},\Theta_{\infty,\updownarrow})=(\Theta^{[3]}_{0,\mathrm{F}}(n,m),\Theta^{[3]}_{\infty,\mathrm{F}}(n,m)).
%\\
%(\Theta_0,\Theta_\infty)=(\Theta^{[1]}_{0,\mathrm{gO}}(m,n),\Theta^{[1]}_{\infty,\mathrm{gO}}(m,n))&\quad\implies\quad
%(\Theta_{0,\updownarrow},\Theta_{\infty,\updownarrow})=(\Theta^{[2]}_{0,\mathrm{gO}}(n,m),\Theta^{[2]}_{\infty,\mathrm{gO}}(n,m))\\
%(\Theta_0,\Theta_\infty)=(\Theta^{[3]}_{0,\mathrm{gO}}(m,n),\Theta^{[3]}_{\infty,\mathrm{gO}}(m,n))&\quad\implies\quad (\Theta_{0,\updownarrow},\Theta_{\infty,\updownarrow})=(\Theta^{[3]}_{0,\mathrm{gO}}(n,m),\Theta^{[3]}_{\infty,\mathrm{gO}}(n,m)).
\end{split}
\endeq
Therefore, as the rational solution of \eqref{p4} for given parameters is unique, \eqref{eq:rotation-symmetry} implies that
\eq
\begin{split}
u^{[2]}_{\mathrm{F}}(x;m,n)&=\ii u^{[1]}_{\mathrm{F}}(-\ii x;n,m)\\
u^{[3]}_{\mathrm{F}}(x;m,n)&=\ii u^{[3]}_{\mathrm{F}}(-\ii x;n,m).
%\\
%u^{[2]}_\mathrm{gO}(x;m,n)&=\ii u^{[1]}_\mathrm{gO}(-\ii x;n,m)\\
%u^{[3]}_\mathrm{gO}(x;m,n)&=\ii u^{[3]}_\mathrm{gO}(-\ii x;n,m).
\end{split}
\label{eq:symmetry-1-2}
\endeq
Since this shows that the rational solutions $u_\mathrm{gH}^{[2]}(x;m,n)$ and $u_\mathrm{gO}^{[2]}(x;m,n)$ are trivially related to $u_\mathrm{gH}^{[1]}(x;n,m)$ and $u_\mathrm{gO}^{[1]}(x;n,m)$ respectively, it is sufficient to study the rational solutions of types $1$ and $3$ in the gH and gO families.

The rational solutions of type $1$ and $3$ are in turn related by a more complicated nonisomonodromic B\"acklund transformation that we denote by  
$\mathcal{S}_\tw$ with action
\eq
\mathcal{S}_\tw(\Theta_0,\Theta_\infty,u(x))=(\Theta_{0,\tw},\Theta_{\infty,\tw},u_\tw(x)):=\left(-\tfrac{1}{2}(\Theta_0+\Theta_\infty),\tfrac{3}{2}\Theta_0-\tfrac{1}{2}\Theta_\infty + 1,\frac{u'(x)}{2u(x)}-\frac{2\Theta_0}{u(x)}-x-\frac{1}{2}u(x)\right).
\label{eq:Baecklund-3-to-1}
\endeq
This is a version of the transformation of Lukashevich \cite{Lukashevich:1967} and Gromak \cite{Gromak:1987} denoted by $\widetilde{W}$ in \cite[Section 2]{Clarkson:2006} and by $\mathscr{T}_1^\pm$ in \cite[\S 32.7(iv)]{DLMF}.  It is interesting to note that the induced action on the $(\Theta_0,\Theta_\infty)$-plane has unit Jacobian; hence area and orientation are preserved.  In fact, setting
\eq
\mathbf{D}(3^{\frac{1}{4}}):=\bpm 3^{\frac{1}{4}} & 0\\ 0 & 3^{-\frac{1}{4}}\epm\quad \text{and}\quad\mathbf{R}(\varphi):= \bpm\cos(\varphi) & -\sin(\varphi)\\\sin(\varphi) & \cos(\varphi)\epm
\endeq
(so $\mathbf{D}(3^\frac{1}{4})$ is the matrix of stretching by $3^\frac{1}{4}$ along the $\Theta_0$ axis and compression by $3^{-\frac{1}{4}}$ along the $\Theta_\infty$ axis, and $\mathbf{R}(\varphi)$ is  the matrix of rigid rotation about the origin by $\varphi$ radians), we have 
\eq
\mathcal{S}_\tw:\left[\bpm\Theta_0\\\Theta_\infty\epm - \bpm \tfrac{1}{6}\\\tfrac{1}{2}\epm\right]\mapsto  \bpm-\tfrac{1}{2}\\\tfrac{1}{2}\epm+\mathbf{D}(3^{\frac{1}{4}})^{-1}\mathbf{R}(\tfrac{2}{3}\pi)\mathbf{D}(3^{\frac{1}{4}})\left[\bpm\Theta_0\\\Theta_\infty\epm-\bpm\tfrac{1}{6}\\\tfrac{1}{2}\epm\right].
\endeq
Since $(\tfrac{1}{6},\tfrac{1}{2})$ is the common vertex of all sectors $W^{[j]\pm}$,
from this form it is easy to see that $\mathcal{S}_\tw$ maps the sector $W^{[3]\pm}$ onto the translation by the lattice vector $(-\tfrac{1}{2},\tfrac{1}{2})$ of the sector $W^{[1]\mp}$.  Moreover, $\mathcal{S}_\tw$ is an isomorphism of $\Lambda_\mathrm{gH}^{[3]+}$ onto $\Lambda_\mathrm{gH}^{[1]-}$, as well as an isomorphism of $\Lambda_\mathrm{gO}$ onto itself.  For either family $\mathrm{F}=\mathrm{gH}$ or $\mathrm{F}=\mathrm{gO}$,
\eq
(\Theta_0,\Theta_\infty)=(\Theta_{0,\mathrm{F}}^{[3]}(m,n),\Theta_{\infty,\mathrm{F}}^{[3]}(m,n))\quad\implies\quad
(\Theta_{0,\tw},\Theta_{\infty,\tw})=(\Theta_{0,\mathrm{F}}^{[1]}(m,n+1),\Theta_{\infty,\mathrm{F}}^{[1]}(m,n+1))
\endeq
so again by uniqueness of the rational solution of \eqref{p4} for given parameters,
\eq
u(x)=u^{[3]}_\mathrm{F}(x;m,n)\quad\implies\quad u_\tw(x)=u^{[1]}_\mathrm{F}(x;m,n+1).
\label{eq:symmetry-1-from-3}
\endeq
Although the rational solutions of types $1$ and $3$ in each family are therefore related explicitly by $\mathcal{S}_\tw$, the explicit expression for $u_\tw(x)$ in terms of $u(x)$ written in \eqref{eq:Baecklund-3-to-1} is not convenient for the study of $u_\tw(x)$ when the parameters are large, even if $u(x)$ is understood with some detail (estimates of derivatives of error terms would be required, for instance).  Therefore we will not use $\mathcal{S}_\tw$ directly; however we will use it indirectly to show that one may extract $u(x)$ and $u_\tw(x)$ from the same Riemann-Hilbert problem by formul\ae\ of comparable complexity, neither of which requires differentiation (see \eqref{eq:u-ucirc} below).  

Another useful but elementary symmetry of \eqref{p4} is the Schwarz symmetry $\mathcal{S}_*$ defined by
\eq
\mathcal{S}_*(\Theta_0,\Theta_\infty,u(x)):=(\Theta_0^*,\Theta_\infty^*,u(x^*)^*).
\endeq
This symmetry combines with iteration of the Boiti-Pempinelli symmetry $\mathcal{S}_\updownarrow$ to yield the following.

\begin{proposition}
Every rational solution $u(x)$ of the Painlev\'e-IV equation \eqref{p4} satisfies
\eq
u(-x)=-u(x),\quad u(x^*)=u(x)^*,\quad\text{and}\quad u(-x^*)=-u(x)^*.
\label{eq:u-reflection-symmetries}
\endeq
In particular, every rational solution $u(x)$ has either a pole or a zero at $x=0$, is real for real $x$ and imaginary for imaginary $x$, and is determined by its values in the closed first quadrant $0\le\arg(x)\le\tfrac{1}{2}\pi$.
\label{prop:FirstQuadrant}
\end{proposition}
\begin{proof}
Noting that $\mathcal{S}_\updownarrow$ is an involution on the parameter space that preserves rationality of $u(x)$, and  using uniqueness of the rational solution for each $(\Theta_0,\Theta_\infty)\in\Lambda_\mathrm{gH}\sqcup\Lambda_\mathrm{gO}$ we deduce that every rational solution $u(x)$ of \eqref{p4} is odd:  $u(-x)=-u(x)$.  On the other hand, $\mathcal{S}_*$ fixes the (real) parameters $(\Theta_0,\Theta_\infty)$ of any rational solution and preserves rationality so by uniqueness every rational solution is Schwarz-symmetric:  $u(x^*)=u(x)^*$; the relation $u(-x^*)=-u(x)^*$ then follows from odd symmetry.  
\end{proof}
It follows that only odd powers of $x$ appear in the power series expansion of any rational solution $u(x)$ about $x=0$.  This allows one to determine sufficiently many terms from the four possible leading terms $\pm x^{-1}$ and $\pm 4\Theta_0 x$ for given $(\Theta_0,\Theta_\infty)$ admitting a rational solution to apply the four elementary isomonodromic B\"acklund transformations $u(x)\mapsto u_\nearrow(x)$, $u(x)\mapsto u_\searrow(x)$, $u(x)\mapsto u_\nwarrow(x)$, and $u(x)\mapsto u_\swarrow(x)$ (see Appendix~\ref{sec:sub-Baecklund}) and deduce the leading term of the expansion at $x=0$ of the image function, also a rational solution of \eqref{p4} for a nearest-neighbor point in the same parameter lattice (depending on the family).  Therefore, starting from any one point in $\Lambda_\mathrm{gH}^{[1]-}$, $\Lambda_\mathrm{gH}^{[2]-}$, $\Lambda_\mathrm{gH}^{[3]+}$, or $\Lambda_\mathrm{gO}$, one can prove the following by induction.
\begin{proposition}
In the limit $x\to 0$, the leading terms of $u_\mathrm{F}^{[j]}(x;m,n)$, $\mathrm{F}=\mathrm{gH}$ or $\mathrm{F}=\mathrm{gO}$ and $j=1,2,3$, depend only on the type $j$ and the parity of the indices $(m,n)$ as follows:
\label{prop:BehaviorAtOrigin}
\end{proposition}
\eq
\begin{tabular}{|c||c|c|}
\hline
\shortstrut$u^{[1]}_\mathrm{F}(x;m,n)$  & $m$ even & $m$ odd \\
\hline\hline
$n$ even & \shortstrut $-4\Theta_{0,\mathrm{F}}^{[1]}(m,n)x+\bo(x^3)$ & $4\Theta_{0,\mathrm{F}}^{[1]}(m,n)x+\bo(x^3)$\\
\hline
$n$ odd  & \shortstrut $x^{-1}+\bo(x)$ & $-x^{-1}+\bo(x)$ \\
\hline
\end{tabular}
\endeq
\eq
\begin{tabular}{|c||c|c|}
\hline
\shortstrut $u^{[2]}_\mathrm{F}(x;m,n)$  & $m$ even & $m$ odd \\
\hline\hline
$n$ even & \shortstrut $-4\Theta_{0,\mathrm{F}}^{[2]}(m,n)x+\bo(x^3)$ & $-x^{-1}+\bo(x)$\\
\hline
$n$ odd  & \shortstrut $4\Theta_{0,\mathrm{F}}^{[2]}(m,n)x+\bo(x^3)$ & $x^{-1}+\bo(x)$ \\
\hline
\end{tabular}
\endeq
\eq
\begin{tabular}{|c||c|c|}
\hline
\shortstrut $u^{[3]}_\mathrm{F}(x;m,n)$  & $m$ even & $m$ odd \\
\hline
\hline
$n$ even & \shortstrut $-4\Theta_{0,\mathrm{F}}^{[3]}(m,n)x+\bo(x^3)$ & $x^{-1}+\bo(x)$\\
\hline
$n$ odd  & \shortstrut $-x^{-1}+\bo(x)$ & $4\Theta_{0,\mathrm{F}}^{[1]}(m,n)x+\bo(x^3)$ \\
\hline
\end{tabular}
\endeq

\subsection{Scaling formalism}
\label{sec:scaling}
%\textcolor{red}{Rewrite this part later for the more general situation in which $\Theta_0$ is large and positive, defining $\Theta_0=-\sigma M$, $\sigma=\mathrm{sign}(-\Theta_0)$, and defining $c$ by $\Theta_\infty=-c\Theta_0$.}
We consider the parameters $\Theta_0$ and $\Theta_\infty$ to be large, of proportional magnitude.  Therefore, we take $T>0$ to be a large parameter, and we assume that for $s=\pm 1$ and $\kappa\in\mathbb{R}$ fixed,
\eq
\Theta_0=sT\quad\text{and}\quad\Theta_\infty=-\kappa T,\quad T>0,\quad \kappa\in\mathbb{R}.
\label{eq:Thetas-scaling}
\endeq
%Fixing $\kappa$\footnote{Assuming that $\alpha$ and $\beta$ correspond to the generalized Okamoto type I family and that $\Theta_0=\tfrac{1}{6}-\tfrac{1}{2}m$ while $m$ and $n$ are both large integers with $n>-\tfrac{1}{2}m$ for $m>0$ or $n<-\tfrac{1}{2}m$ for $m<0$, we see easily that 
%\[
%\Theta_\infty=\frac{1}{2}-n-\frac{1}{2}m
%\]
%and therefore
%\[
%\kappa=-\frac{\Theta_\infty}{|\Theta_0|}=\frac{2n+m-1}{|m-\tfrac{1}{3}|}.
%\]
%Hence $T$ is the asymptotic value of $\tfrac{1}{2}|m|$ and $\kappa$ is the asymptotic value of $2n/m+1$ for $m>0$ or $-2n/m-1$ for $m<0$.  Note that $\kappa\in [0,1)\cup (1,+\infty)$ corresponds to the $m>0$ case, while $\kappa\in (-\infty,-1)\cup (-1,0]$ corresponds to the $m<0$ case.}, we formally analyze the Painlev\'e-IV equation \eqref{eq:PIV-Thetas} in the limit $T\to+\infty$.  
With this scaling, we formally analyze the Painlev\'e-IV equation \eqref{p4} in the limit $T\to +\infty$.
To obtain a dominant balance, we will also scale $u$ and $x$ as follows:
\eq
u=T^{\frac{1}{2}}U\quad\text{and, for fixed $y_0\in\mathbb{C}$,}\quad x=T^{\frac{1}{2}}y_0+T^{-\frac{1}{2}}\zeta
\label{eq:BasicScalings}
\endeq
where we view $U$ and $\zeta$ as the new dependent and independent variables respectively.  Then it is easy to see that the Painlev\'e-IV equation takes the form
\eq
U_{\zeta\zeta}=\frac{(U_\zeta)^2}{2U}+\frac{3}{2}U^3 + 4y_0U^2 + (2y_0^2+4\kappa)U-\frac{8}{U} + \bo(T^{-1})
\endeq
Letting $T\to\infty$, we formally obtain an autonomous equation governing a formal approximation $\dot{U}(\zeta)$ of $U(\zeta)$, in which $y_0\in\mathbb{C}$ and $\kappa\in\mathbb{R}\setminus\{-1,1\}$ appear as parameters:
\eq
\dot{U}_{\zeta\zeta}=\frac{(\dot{U}_\zeta)^2}{2\dot{U}}+\frac{3}{2}\dot{U}^3+4y_0\dot{U}^2+(2y_0^2+4\kappa)\dot{U}-\frac{8}{\dot{U}}.
\label{eq:Approximating-ODE-Second-Order}
\endeq

\subsubsection{Equilibrium solutions of the autonomous approximating equation and their branch points}
\label{sec:equilibrium}
The approximating equation \eqref{eq:Approximating-ODE-Second-Order} has equilibrium solutions $\dot{U}(\zeta)=U_0$ (independent of $\zeta$) that are roots of the quartic equation
\eq
\frac{3}{2}U_0^4+4y_0U_0^3+(2y_0^2+4\kappa)U_0^2-8=0.
\label{eq:equilibrium}
\endeq
This relation is invariant under the cyclic group of order $4$ generated by an analogue of the Boiti-Pempenelli symmetry $\mathcal{S}_\updownarrow$ introduced in \eqref{eq:rotation-symmetry} with action
$\mathcal{S}_\updownarrow(\kappa,y_0,U_0)\mapsto (-\kappa,\ii y_0,\ii U_0)$.
When $y_0$ is large there are four distinct roots that we denote by $U_0=U_{0,\mathrm{gH}}^{[j]}(y_0;\kappa)$, $j=1,2,3$, and $U_0=U_{0,\mathrm{gO}}(y_0;\kappa)$.  These are all analytic functions of $y_0$ large with asymptotic behavior
\eq
\begin{split}
U_{0,\mathrm{gH}}^{[1]}(y_0;\kappa)&=2y_0^{-1}(1+o(1)),\\
U_{0,\mathrm{gH}}^{[2]}(y_0;\kappa)&=-2y_0^{-1}(1+o(1)),\\ 
U_{0,\mathrm{gH}}^{[3]}(y_0;\kappa)&=-2y_0(1+o(1)),\quad\text{and}\\
U_{0,\mathrm{gO}}(y_0;\kappa)&=-\tfrac{2}{3}y_0(1+o(1)),\quad y_0\to\infty.
\end{split}
\label{eq:equilibria-near-infinity}
\endeq
Since these functions are all distinct for large $y_0$, $\mathcal{S}_\updownarrow$ acts on them by
\eq
\begin{split}
\mathcal{S}_\updownarrow U_{0,\mathrm{gH}}^{[1]}(y_0;\kappa)&:=\ii U_{0,\mathrm{gH}}^{[1]}(-\ii y_0;-\kappa)=U_{0,\mathrm{gH}}^{[2]}(y_0;\kappa)\\
\mathcal{S}_\updownarrow U_{0,\mathrm{gH}}^{[2]}(y_0;\kappa)&:=\ii U_{0,\mathrm{gH}}^{[2]}(-\ii y_0;-\kappa)=U_{0,\mathrm{gH}}^{[1]}(y_0;\kappa)\\
\mathcal{S}_\updownarrow U_{0,\mathrm{gH}}^{[3]}(y_0;\kappa)&:=\ii U_{0,\mathrm{gH}}^{[3]}(-\ii y_0;-\kappa)=U_{0,\mathrm{gH}}^{[3]}(y_0;\kappa)\\
\mathcal{S}_\updownarrow U_{0,\mathrm{gO}}(y_0;\kappa)&:=\ii U_{0,\mathrm{gO}}(-\ii y_0;-\kappa)=U_{0,\mathrm{gO}}(y_0;\kappa).
\end{split}
\label{eq:equilibrium-symmetry}
\endeq
These relations also imply that, for fixed $\kappa\in\mathbb{R}\setminus\{-1,1\}$, all four equilibrium branches are odd functions of $y_0$.
%If we suppose that an equilibrium solution approximates $U(\zeta)$ when $y_0$ is large (i.e., near $x=\infty$), then from the known large-$x$ asymptotics of the different types of rational solutions of the Painlev\'e-IV equation we see that the first three equilibria would be expected to approximate the rational solutions obtained from generalized Hermite polynomials (of types $(I)$, $(II)$, and $(III)$ respectively \textcolor{red}{give a reference or explain the generalized Hermite taxonomy in the introduction}), and only the last equilibrium behaving like $-\tfrac{2}{3}y_0+O(1)$ might be expected to approximate the Okamoto rational solutions.  

%\subsubsection{Branch points of the equilibrium solutions}
If we select any of the equilibria and try to analytically continue the solution to finite values of $y_0$ we will only encounter any obstruction at branch points of $U_0$; these are precisely the values of $y_0$ for which there are double roots $U_0$ of \eqref{eq:equilibrium}.  So together with \eqref{eq:equilibrium} we consider its derivative with respect to $U_0$ and eliminate $U_0$ between the two equations to obtain a polynomial discriminant in $y_0$ whose roots are the desired branch points.  
By standard calculations, one finds that the discriminant is
%In this case, the derivative is a cubic with one root at $U_0=0$ (which is inconsistent with \eqref{eq:equilibrium}) and the other two given via the quadratic formula by
%\eq
%U_0=-y_0\pm\sqrt{\frac{1}{3}(y_0^2-4\kappa)}\quad\text{for double roots of \eqref{eq:equilibrium}.}
%\endeq
%Substituting into \eqref{eq:equilibrium}, isolating the odd powers of the square root and squaring to remove the root yields the discriminant in the form
\eq
B(y_0;\kappa):=y_0^8-24(\kappa^2+3)y_0^4-64\kappa(\kappa+3)(\kappa-3)y_0^2-48(\kappa^2+3)^2=0\quad\text{for branch points $y_0$ of equilibria.}
\label{eq:branch-points}
\endeq

Given $\kappa\in\mathbb{R}\setminus\{-1,1\}$, this equation has eight roots in the complex $y_0$-plane (the discriminant, i.e., the polynomial resultant of $B(\cdot;\kappa)$ and $B'(\cdot;\kappa)$, whose vanishing is equivalent to the existence of double roots of $B(\cdot;\kappa)$, is proportional to $(\kappa^2-1)^8(\kappa^2+3)^2$).  The branch point set is obviously symmetric with respect to
reflection through the real and imaginary axes.  Observe that $B(y_0;1)=(y_0^2-4)^3(y_0^2+12)$, which has a conjugate pair of simple purely imaginary roots at $y_0=\pm\ii\sqrt{12}$ and two triple real roots at $y_0=\pm 2$.  If we consider $\kappa=1+\epsilon$ for small positive $\epsilon$, then by symmetry there will be again a pair of purely imaginary simple roots of $B(y_0;1+\epsilon)$, and by appropriate rescaling of $y_0\mp 2$ with $\epsilon$ one finds that each triple root splits a triad of three nearby simple roots of the form $y_0=\pm 2 (1+\ee^{\frac{2\ii\pi k}{3}}\epsilon^{\frac{2}{3}}(864)^{\frac{1}{3}}+\bo(\epsilon))$ as $\epsilon\downarrow 0$, where $k=0,\pm 1$.  In particular, this shows that for $\kappa$ just greater than $1$, $B(y_0;\kappa)$ has eight simple roots comprising two opposite purely real and purely imaginary pairs in addition to a complex quartet of roots symmetric with respect to reflection through the real and imaginary axes.  Since non-simple roots of $B(\cdot;\kappa)$ can only occur for real $\kappa=\pm 1$, there can be no collisions of roots of $B(\cdot;\kappa)$ as $\kappa$ increases from $1$, and together with the reflection symmetry of the roots in the real and imaginary axes this implies that for all $\kappa>1$ the roots of $B(\cdot;\kappa)$ are all simple, with opposite real and imaginary pairs in addition to a symmetric complex quartet of roots, just as for $\kappa=1+\epsilon$ with $\epsilon>0$ small.  If instead we consider $-1< \kappa<1$, a very similar argument goes through; now one should replace $\epsilon$ with $-\epsilon<0$ small and negative to perturb from $\kappa=1$ in the negative direction, and the perturbation analysis produces an extra factor of $\ee^{\frac{2\ii\pi}{3}}$ on the subleading term $y_0\approx \pm 2$ which of course just means re-indexing $k$.  So we again have the same  triads of nearby roots, and since there can be no non-simple roots of $B(\cdot;\kappa)$ for $-1<\kappa<1$ the same picture persists throughout this interval as well.  
%The only degeneration occurs exactly at $\kappa=1$, with the rectangle collapsing to a real line segment and the triangles at the real endpoints collapsing to triple branch points.  
Finally, we observe that $B(y_0;\kappa)=B(\ii y_0;-\kappa)$, so the branch point configuration for $\kappa<-1$ follows immediately from that for $-\kappa>1$ by rotation in the complex $y_0$-plane by $\tfrac{1}{2}\pi$.

It follows in particular that for all $\kappa\in\mathbb{R}\setminus\{-1,1\}$, in each of the four open half-planes $\pm \mathrm{Re}(y_0)>0$, $\pm\mathrm{Im}(y_0)>0$, there is a triad of simple roots of $B(\cdot;\kappa)$ symmetric with respect to reflection through the real or imaginary axis bisecting the half-plane in question, and further characterized by the following additional remarkable property.
\begin{proposition}
Let $\kappa\in\mathbb{R}\setminus\{-1,1\}$.  The triad of roots $y_0$ of $B(y_0;\kappa)=0$ in each half-plane ($\pm\mathrm{Re}(y_0)>0$ or $\pm\mathrm{Im}(y_0)>0$) form the vertices of an equilateral triangle.
\label{prop:triangles}
\end{proposition}
We give the proof in Appendix~\ref{app:Equilateral}.

\subsubsection{Nonequilibrium solutions of the autonomous approximating equation}
\label{sec:nonequilibrium}
More generally we may consider nonequilibrium solutions of the model differential equation 
\eqref{eq:Approximating-ODE-Second-Order}.
Using the integrating factor $\dot{U}_\zeta/\dot{U}$, \eqref{eq:Approximating-ODE-Second-Order} implies %and collecting the terms to the same side of the equation gives
\eq
\begin{split}
0&=
\frac{\dot{U}_\zeta\dot{U}_{\zeta\zeta}}{\dot{U}}-\frac{(\dot{U}_\zeta)^3}{2\dot{U}^2}-\frac{3}{2}\dot{U}^2\dot{U}_\zeta-4y_0\dot{U}\dot{U}_\zeta - (2y_0^2+4\kappa)\dot{U}_\zeta+\frac{8\dot{U}_\zeta}{\dot{U}^2}\\
&=
\frac{\dd}{\dd\zeta}\left[\frac{(\dot{U}_\zeta)^2}{2\dot{U}}-\frac{1}{2}\dot{U}^3-2y_0\dot{U}^2-(2y_0^2+4\kappa)\dot{U}-\frac{8}{\dot{U}}\right].
\end{split}
\endeq
Therefore, if $E$ denotes an integration constant
%\eq
%\frac{(\dot{U}_\zeta)^2}{2\dot{U}}-\frac{1}{2}\dot{U}^3-2y_0\dot{U}^2-(2y_0^2+4\kappa)\dot{U}-\frac{8}{\dot{U}}=E,
%\endeq
%in other words, $\dot{U}=f(\zeta)$ where
then $\dot{U}=f(\zeta)$ is a solution of the first-order second-degree equation
\eq
\left(\frac{\dd f}{\dd\zeta}\right)^2 = P(f),\quad\text{where}\quad P(f)=P(f;y_0,\kappa,E):=f^4+4y_0f^3 +4(y_0^2+2\kappa)f^2 + 2Ef + 16.
\label{eq:elliptic-ODE}
\endeq
Since $P(f)$ is a quartic polynomial, the nonequilibrium solutions $\dot{U}=f(\zeta)$ are clearly elliptic functions with modulus depending on $y_0$, $\kappa$, and $E$, having two fundamental periods 
\eq
Z_\mathfrak{a}:=\oint_\mathfrak{a}\frac{\dd f}{\sqrt{P(f)}}\quad\text{and}\quad
Z_\mathfrak{b}:=\oint_\mathfrak{b}\frac{\dd f}{\sqrt{P(f)}}
\label{eq:elliptic-periods}
\endeq
where $\mathfrak{a}$ and $\mathfrak{b}$ are independent cycles on the genus-$1$ Riemann surface of the equation $w^2=P(z)$, forming a basis for its homology group.  In other words, $f(\zeta+Z_\mathfrak{a})=f(\zeta+Z_\mathfrak{b})=f(\zeta)$, and $Z_\mathfrak{a}$ and $Z_\mathfrak{b}$ are linearly independent over the real numbers.  
It is easy to see that every nonconstant solution \eqref{eq:elliptic-ODE} is an elliptic function with only simple zeros (with derivative $\pm 4$) and at worst simple poles (with residue $\pm 1$).  Moreover, every solution has one zero with each sign of derivative and one pole of each residue within each period parallelogram.  Since \eqref{eq:elliptic-ODE} is a first-order second-degree autonomous equation, if $f(\zeta)$ denotes the unique solution of the differential equation \eqref{eq:elliptic-ODE} satisfying the initial conditions $f(0)=0$ and $f'(0)=4$, then every non-constant solution of \eqref{eq:elliptic-ODE} can be written in the form $f(\zeta-\zeta_0)$ for a uniquely determined constant phase shift $\zeta_0\in\mathbb{C}$ that one may specify modulo integer linear combinations of the periods $Z_\mathfrak{a}$ and $Z_\mathfrak{b}$.

\begin{remark}
The quartic polynomial $P(f)=P(f;y_0,\kappa,E)$ has been motivated here by formal asymptotic analysis of the Painlev\'e-IV equation and hence $f$ plays the role of the dependent variable.  However as will be seen in Section~\ref{sec:SpectralCurve} it also defines a \emph{spectral curve} in which $f$ plays instead the role of a rescaled auxiliary spectral parameter $z$ from the Lax pair representation of the Painlev\'e-IV equation; see \eqref{eq:spectral-curve}.  The same polynomial again appears in the anharmonic quantum oscillator theory of Masoero and Roffelsen \cite{MasoeroR18,MasoeroR19}, in which the leading term $V(\lambda)$ in the anharmonic potential is proportional to $\lambda^{-2}P(\lambda)$.  Here again the variable $\lambda$ plays an auxiliary role not directly tied to an approximate solution of the Painlev\'e-IV equation.
\label{rem:SamePolynomial}
\end{remark}

We prove in this paper the accuracy of the elliptic function approximation; in particular this applies in the special case that $y_0=0$ and $\zeta$ is bounded.  To prove this special case one must take $E=0$ for the integration constant (see Proposition~\ref{prop:E} below), and then it turns out that the approximation captures exactly the pole or zero of the rational solution $u(x)$ that necessarily lies at the origin ($x=0$ corresponds to $\zeta=0$ in the rescaled independent variable) according to Proposition~\ref{prop:FirstQuadrant}. See Remark~\ref{rem:exact-at-origin} below. Therefore, in this case we can determine the phase shift $\zeta_0$ by enforcing the property that $f(\zeta-\zeta_0)$ have a zero with the same sign of derivative or a pole with the same residue
at $\zeta=0$ as does the actual rational solution, as described in Proposition~\ref{prop:BehaviorAtOrigin}.
We start from \eqref{eq:elliptic-ODE} with $y_0=0$ and $E=0$, and we first solve this equation subject to $f(0)=0$ and $f'(0)=4$ to obtain $f(\zeta)$ explicitly in terms of Jacobi elliptic functions (in terms of the notation of \cite[Chapter 22]{DLMF} we prefer to write, eg., $\mathrm{sn}(z\vert\mathfrak{m})$ in place of $\mathrm{sn}(z,k)$ where $\mathfrak{m}=k^2$), and express its pole and zero lattices in terms of the complete elliptic integrals of the first kind
\eq
\mathbb{K}(\mathfrak{m}):=\int_0^{\tfrac{1}{2}\pi}\frac{\dd\theta}{\sqrt{1-\mathfrak{m}\sin(\theta)}}, \quad \mathbb{K}'(\mathfrak{m}):=\mathbb{K}(1-\mathfrak{m}).
\endeq
This calculation depends only on whether $\kappa<-1$, $\kappa\in (-1,1)$, or $\kappa>1$ holds, and yields the following results.
\begin{itemize}
\item If $\kappa<-1$ (i.e., for $(\Theta_0,\Theta_\infty)\in\Lambda_\mathrm{gH}^{[1]-}\cup(\Lambda_\mathrm{gO}\cap W^{[1]-})\cup(\Lambda_\mathrm{gO}\cap W^{[2]+})$), then 
\eq
f(\zeta) = 2\mathfrak{m}^\frac{1}{4}\mathrm{sn}\left(2\mathfrak{m}^{-\frac{1}{4}}\zeta \big\vert \mathfrak{m} \right),\quad \mathfrak{m}:=-1+2\kappa^2+2\kappa\sqrt{\kappa^2-1}\in (0,1).
\label{eq:f-m-kappa-lt-m1}
\endeq
It is known that $\mathrm{sn}(x|\mathfrak{m})$ has zeros at $x=2j\mathbb{K}(\mathfrak{m})+2k\ii\mathbb{K}^\prime(\mathfrak{m})$ 
and poles at $x=2j\mathbb{K}(\mathfrak{m})+(2k+1)\ii\mathbb{K}^\prime(\mathfrak{m})$ 
($j,k\in\mathbb{Z}$).  Since $\mathbb{K}(\mathfrak{m})$ and $\mathbb{K}^\prime(\mathfrak{m})$ 
are both purely real, as is the scaling factor 
$2\mathfrak{m}^{-\frac{1}{4}}$, $f(\zeta)$ has rows of zeros 
alternating with rows of poles parallel to the real axis.
\item If $\kappa>1$ (i.e., for $(\Theta_0,\Theta_\infty)\in\Lambda_\mathrm{gH}^{[2]-}\cup(\Lambda_\mathrm{gO}\cap W^{[1]+})\cup(\Lambda_\mathrm{gO}\cap W^{[2]-})$), then
\eq
f(\zeta) = 2(1-\mathfrak{m})^\frac{1}{4}\mathrm{sc}\left(2(1-\mathfrak{m})^{-\frac{1}{4}}\zeta \big\vert \mathfrak{m} \right),\quad \mathfrak{m}:=2-2\kappa^2+2\kappa\sqrt{\kappa^2-1}\in (0,1).
\label{eq:f-m-kappa-gt-p1}
\endeq
The function $\mathrm{sc}(x|\mathfrak{m})$ has zeros at $x=2j\mathbb{K}(\mathfrak{m})+2k\ii\mathbb{K}^\prime(\mathfrak{m})$ 
and poles at $x=(2j+1)\mathbb{K}(\mathfrak{m})+2k\ii\mathbb{K}^\prime(\mathfrak{m})$ ($j,k\in\mathbb{Z}$).  Since $\mathbb{K}(\mathfrak{m})$, 
$\mathbb{K}^\prime(\mathfrak{m})$, and the scaling factor $2(1-\mathfrak{m})^{-\frac{1}{4}}$ are all 
purely real, $f(\zeta)$ has columns of zeros alternating with columns 
of poles parallel to the imaginary axis.
\item If $\kappa\in (-1,1)$ (i.e., for $(\Theta_0,\Theta_\infty)\in\Lambda_\mathrm{gH}^{[3]+}\cup(\Lambda_\mathrm{gO}\cap W^{[3]+})\cup(\Lambda_\mathrm{gO}\cap W^{[3]-})$), then
\eq
f(\zeta) = 2\ee^{-\frac{\ii\pi}{4}}((1-\mathfrak{m})\mathfrak{m})^\frac{1}{4}\mathrm{sd}\left(2\ee^\frac{\ii\pi}{4}((1-\mathfrak{m})\mathfrak{m})^{-\frac{1}{4}}\zeta \big\vert \mathfrak{m} \right),\quad \mathfrak{m}:=\frac{1}{2}-\frac{\ii\kappa}{2\sqrt{1-\kappa^2}}=1-\mathfrak{m}^*.
\label{eq:f-m-kappa-0}
\endeq
This is the only case in which the elliptic modulus $\mathfrak{m}$ is complex, in which case we use the principal branch square roots to interpret $\mathbb{K}(\mathfrak{m})$ and $\mathbb{K}'(\mathfrak{m})$; therefore as $\mathfrak{m}^*=1-\mathfrak{m}$ we also have $\mathbb{K}'(\mathfrak{m})=\mathbb{K}(\mathfrak{m})^*$.
The elliptic function $\mathrm{sd}(x|\mathfrak{m})$ has 
zeros at $x=2j\mathbb{K}(\mathfrak{m})+2k\ii\mathbb{K}^\prime(\mathfrak{m})$ and poles at 
$x=(2j+1)\mathbb{K}(\mathfrak{m})+(2k+1)\ii\mathbb{K}^\prime(\mathfrak{m})$ ($j,k\in\mathbb{Z}$).
Since $\mathrm{arg}(\mathbb{K}(\mathfrak{m})+\ii\mathbb{K}^\prime(\mathfrak{m}))=\tfrac{1}{4}\pi$ and 
$\mathrm{arg}(-\mathbb{K}(\mathfrak{m})+\ii\mathbb{K}^\prime(\mathfrak{m}))=\tfrac{3}{4}\pi$, the zeros 
and poles of $f(\zeta)$ form a ``checkerboard'' pattern with respective lattices spanned by basis vectors 
parallel to the coordinate axes and shifted by a half-period in each direction with respect to one another.
\end{itemize}
In all three cases, the theoretically predicted pattern qualitatively matches what one sees near the origin in the respective plots shown in Figures~\ref{fig:theta-map-gH}--\ref{fig:theta-map-gO}.
%\eq
%(\dot{U}_\zeta)^2 = \dot{U}^4+4y_0\dot{U}^3+2(2y_0^2+4\kappa)\dot{U}^2+2E\dot{U}+16
%\endeq
%and set $y_0=0$ so $E=0$:
%\eq
%(\dot{U}_\zeta)^2 = \dot{U}^4 + 8\kappa\dot{U}^2 + 16.
%\endeq
%For each sector $W^{[1]\pm}$, $W^{[2]\pm}$, $W^{[3]\pm}$ in the 
%$\Theta_0$-$\Theta_\infty$ plane and each choice of the parity of $m$ and $n$, 
%we record the explicit solution $\dot{U}(\zeta) = f(\zeta-\zeta_0)$, where 
%$f(0)=0$ and $f'(0)=4$.  In each case, given the elliptic modulus 
%$\mathfrak{m}$, we use the complete elliptic integrals of the first kind 
%\eq
%\mathbb{K}(\mathfrak{m}):=\int_0^{\pi/2}\frac{\dd\theta}{\sqrt{1-\mathfrak{m}\sin(\theta)}}, \quad \mathbb{K}'(\mathfrak{m}):=\mathbb{K}(1-\mathfrak{m}).
%\endeq
It remains to determine the phase shift $\zeta_0$ by ensuring that the correct ``sign'' of pole or zero of $f(\zeta-\zeta_0)$ lies at $\zeta=0$.  Here the results depend not only on the sector of the parameter space shown in Figure~\ref{fig:ThetasPlane} but also on the parity of the indices $(m,n)$ used to parametrize the allowed values of $(\Theta_0,\Theta_\infty)$.  The results are as follows.
\begin{itemize}
\item
Let $\kappa<-1$ and take $\mathfrak{m}\in (0,1)$ as in \eqref{eq:f-m-kappa-lt-m1}. If $(\Theta_{0,\mathrm{F}}^{[1]}(m,n),\Theta_{\infty,\mathrm{F}}^{[1]}(m,n))\in W^{[1]-}$ (for either family $\mathrm{F}=\mathrm{gH}$ or $\mathrm{F}=\mathrm{gO}$), then $\zeta_0$ is given by
\eq
\begin{tabular}{|c||c|c|}
\hline
$\zeta_0$  & $m$ even & $m$ odd \\
\hline
\hline
$n$ even & 0 & \tallstrut $\displaystyle\frac{2\mathbb{K}(\mathfrak{m})}{2\mathfrak{m}^{-\frac{1}{4}}}$ \\
\hline
$n$ odd  & \tallstrut $\displaystyle\frac{\ii\mathbb{K}^\prime(\mathfrak{m})}{2\mathfrak{m}^{-\frac{1}{4}}}$ & $\displaystyle\frac{2\mathbb{K}(\mathfrak{m})+\ii\mathbb{K}^\prime(\mathfrak{m})}{2\mathfrak{m}^{-\frac{1}{4}}}$\\
\hline 
\end{tabular}
\endeq
If instead $(\Theta_{0,\mathrm{gO}}^{[2]}(m,n),\Theta_{\infty,\mathrm{gO}}^{[2]}(m,n))\in W^{[2]+}$, then $\zeta_0$ is given by
\eq
\begin{tabular}{|c||c|c|}
\hline
$\zeta_0$  & $m$ even & $m$ odd \\
\hline
\hline
$n$ even & \tallstrut $\displaystyle\frac{2\mathbb{K}(\mathfrak{m})}{2\mathfrak{m}^{-\frac{1}{4}}}$  & $\displaystyle\frac{2\mathbb{K}(\mathfrak{m})+\ii\mathbb{K}^\prime(\mathfrak{m})}{2\mathfrak{m}^{-\frac{1}{4}}}$ \\
\hline
$n$ odd  & 0 & \tallstrut $\displaystyle\frac{\ii\mathbb{K}^\prime(\mathfrak{m})}{2\mathfrak{m}^{-\frac{1}{4}}}$ \\
\hline
\end{tabular}
\endeq
\item
Let $\kappa>1$ and take $\mathfrak{m}\in (0,1)$ as in \eqref{eq:f-m-kappa-gt-p1}.  If $(\Theta_{0,\mathrm{gO}}^{[1]}(m,n),\Theta_{\infty,\mathrm{gO}}^{[1]}(m,n))\in W^{[1]+}$, then $\zeta_0$ is given by
\eq
\begin{tabular}{|c||c|c|}
\hline
$\zeta_0$  & $m$ even & $m$ odd \\
\hline
\hline
$n$ even & \tallstrut$\displaystyle\frac{2\ii\mathbb{K}^\prime(\mathfrak{m})}{2(1-\mathfrak{m})^{-\frac{1}{4}}}$ & 0 \\
\hline
$n$ odd  & \tallstrut $\displaystyle\frac{\mathbb{K}(\mathfrak{m})+2\ii\mathbb{K}^\prime(\mathfrak{m})}{2(1-\mathfrak{m})^{-\frac{1}{4}}}$ & $\displaystyle\frac{\mathbb{K}(\mathfrak{m})}{2(1-\mathfrak{m})^{-\frac{1}{4}}}$ \\
\hline
\end{tabular}
\endeq
If instead 
$(\Theta_{0,\mathrm{F}}^{[2]}(m,n),\Theta_{\infty,\mathrm{F}}^{[2]}(m,n))\in W^{[2]-}$ (for either family $\mathrm{F}=\mathrm{gH}$ or $\mathrm{F}=\mathrm{gO}$), then $\zeta_0$ is given by
\eq
\begin{tabular}{|c||c|c|}
\hline
$\zeta_0$  & $m$ even & $m$ odd \\
\hline
\hline
$n$ even & 0 & \tallstrut $\displaystyle\frac{\mathbb{K}(\mathfrak{m})}{2(1-\mathfrak{m})^{-\frac{1}{4}}}$ \\
\hline
$n$ odd  & \tallstrut $\displaystyle\frac{2\ii\mathbb{K}^\prime(\mathfrak{m})}{2(1-\mathfrak{m})^{-\frac{1}{4}}}$ & $\displaystyle\frac{\mathbb{K}(\mathfrak{m})+2\ii\mathbb{K}^\prime(\mathfrak{m})}{2(1-\mathfrak{m})^{-\frac{1}{4}}}$ \\
\hline
\end{tabular}
\endeq
\item
Let $\kappa\in (-1,1)$ and take $\kappa\in\mathbb{C}$ as in \eqref{eq:f-m-kappa-0}.  If 
$(\Theta_{0,\mathrm{F}}^{[3]}(m,n),\Theta_{\infty,\mathrm{F}}^{[3]}(m,n))\in W^{[3]+}$ (for either family $\mathrm{F}=\mathrm{gH}$ or $\mathrm{F}=\mathrm{gO}$), then $\zeta_0$ is given by
\eq
\begin{tabular}{|c||c|c|}
\hline
$\zeta_0$  & $m$ even & $m$ odd \\
\hline
\hline
$n$ even & \tallstrut $\displaystyle\frac{2\mathbb{K}(\mathfrak{m})}{2\ee^{\frac{\ii\pi}{4}}((1-\mathfrak{m})\mathfrak{m})^{-\frac{1}{4}}}$ & $\displaystyle\frac{-\mathbb{K}(\mathfrak{m})+\ii\mathbb{K}^\prime(\mathfrak{m})}{2\ee^{\frac{\ii\pi}{4}}((1-\mathfrak{m})\mathfrak{m})^{-\frac{1}{4}}}$ \\
\hline
$n$ odd  & \tallstrut $\displaystyle\frac{\mathbb{K}(\mathfrak{m})+\ii\mathbb{K}^\prime(\mathfrak{m})}{2\ee^{\frac{\ii\pi}{4}}((1-\mathfrak{m})\mathfrak{m})^{-\frac{1}{4}}}$ & 0 \\
\hline
\end{tabular}
\endeq
If instead $(\Theta_{0,\mathrm{gO}}^{[3]}(m,n),\Theta_{\infty,\mathrm{gO}}^{[3]}(m,n))\in W^{[3]-}$, then $\zeta_0$ is given by
\eq
\begin{tabular}{|c||c|c|}
\hline
$\zeta_0$  & $m$ even & $m$ odd \\
\hline
\hline
$n$ even & 0 & \tallstrut $\displaystyle\frac{-\mathbb{K}(\mathfrak{m})+\ii\mathbb{K}^\prime(\mathfrak{m})}{2\ee^{\frac{\ii\pi}{4}}((1-\mathfrak{m})\mathfrak{m})^{-\frac{1}{4}}}$ \\
\hline
$n$ odd  & \tallstrut $\displaystyle\frac{\mathbb{K}(\mathfrak{m})+\ii\mathbb{K}^\prime(\mathfrak{m})}{2\ee^{\frac{\ii\pi}{4}}((1-\mathfrak{m})\mathfrak{m})^{-\frac{1}{4}}}$ & $\displaystyle\frac{2\mathbb{K}(\mathfrak{m})}{2\ee^{\frac{\ii\pi}{4}}((1-\mathfrak{m})\mathfrak{m})^{-\frac{1}{4}}}$ \\
\hline
\end{tabular}
\endeq
\end{itemize}
These results are consistent in the gH case with \cite[Corollary 1]{MasoeroR18}, a rigorous result describing the zeros of $H_{m,n}(x)$ near the origin as a locally regular lattice with spacings determined from complete elliptic integrals.

\begin{remark}
Constant solutions of \eqref{eq:elliptic-ODE} corresponding to simple roots of $P(f)$ are not equilibrium solutions of \eqref{eq:Approximating-ODE-Second-Order}.  However,
the equilibrium condition \eqref{eq:equilibrium} can be rederived by insisting that $P(f)$ have a double root $f=U_0$ and hence eliminating $E$ between $P(U_0)=0$ and $P'(U_0)=0$, leading to
\eq
E=-2U_0^3-6y_0U_0^2-4(y_0^2+2\kappa)U_0.
\label{eq:E-for-double-roots}
\endeq
Another way of putting this is: if $U_0=U_0(y_0;\kappa)$ is a function of $y_0\in\mathbb{C}$ and $\kappa\in \mathbb{R}\setminus\{-1,1\}$ that solves \eqref{eq:equilibrium}, and $E$ is expressed in terms of $y_0$ and $\kappa$ by \eqref{eq:E-for-double-roots}, then $P(f)$ has a double root that persists over the whole domain of definition of $U_0(y_0;\kappa)$, namely $f=U_0(y_0;\kappa)$.
\label{rem:equilibrium-rederive}
\end{remark}

\subsubsection{Painlev\'e-I approximation near branch points}
Suppose now that $y_0$ is one of the eight branch points solving \eqref{eq:branch-points}, and that $U_0$ is a corresponding double root of the equilibrium problem \eqref{eq:equilibrium}.  We wish to examine solutions of the Painlev\'e-IV equation \eqref{p4} that are in a sense close to $T^{\frac{1}{2}}U_0$ for $x$ close to $T^{\frac{1}{2}}y_0$, where we recall that the parameters are large in the sense that \eqref{eq:Thetas-scaling} holds with $T\gg 1$ and $\kappa\in\mathbb{R}\setminus\{-1,1\}$.
%$\Theta_0=-M$ and $\Theta_\infty=cM$ for $c<-1$.  
It turns out that the correct scaling is to write
\eq
u=T^{\frac{1}{2}}(U_0+T^{-\frac{2}{5}}h)\quad\text{and}\quad x=T^{\frac{1}{2}}y_0+T^{-\frac{3}{10}}z
\endeq
for new dependent and independent variables $h$ and $z$, respectively.  We substitute these into \eqref{p4} and use the 
assumption that $U_0$ is a double root of the equilibrium equation \eqref{eq:equilibrium} to remove two suites of terms from the resulting equation.
%fact that the terms constant and linear with respect to $h$ in the terms
%\eq
%\frac{3}{2}u^3 + 4M^{1/2}y_0u^2 +(2My_0^2-4\Theta_\infty)u-\frac{8\Theta_0^2}{u}
%\endeq
%vanish by choice of $y_0$ and $U_0$.  
The result is the formal asymptotic (assuming $h$ and $z$ bounded)
\eq
h_{zz}=\left(\frac{9}{2}U_0+4y_0-\frac{8}{U_0^3}\right)h^2 + 4U_0\left(U_0+y_0\right)z + 2T^{-\frac{1}{5}}U_0 + \bo(T^{-\frac{2}{5}}),\quad T\to\infty.
\endeq
%\textcolor{red}{Check this again.  It is strange that the coefficient of $h^2$ does not have any simple symmetry with respect to $y_0\mapsto \ii y_0$ and $U_0\mapsto \ii U_0$.}
This is essentially a perturbation of the Painlev\'e-I equation.  Indeed, if we rescale the variables by 
\eq
w=\gamma\left(z+\frac{T^{-\frac{1}{5}}}{2(U_0+y_0)}\right)\quad \text{and}\quad h=\delta H,
\endeq
then if $\delta$ and $\gamma$ are chosen so that 
%
%$\delta=4U_0^2\gamma$ and
%\eq
%4U_0^2\gamma^3\left(\frac{9}{2}U_0+4y_0-\frac{8}{U_0}\right)=6
%\endeq
%\textcolor{red}{(should check that the expression in parentheses can't vanish when $U_0$ is a double equilibrium for branch points $y_0$)} we obtain
\eq
\gamma^5 = \frac{9U_0^4+8y_0U_0^3-16}{12 U_0^3}\quad\text{and}\quad
\delta = 4U_0(U_0+y_0)\gamma^{-3},
\endeq
we obtain
\eq
H''(w)=6H(w)^2+w+\bo(T^{-\frac{2}{5}})
\endeq
which puts the Painlev\'e-I approximating equation into canonical form.  Note that $\gamma$ and $\delta$ are well-defined modulo the cyclic group of order $5$ generated by $(\gamma,\delta)\mapsto (\ee^{\frac{2\ii\pi}{5}}\gamma,\ee^{-\frac{6\ii\pi}{5}}\delta)$, for which it suffices to show that $U_0\neq 0$, $U_0+y_0\neq 0$, and $9U_0^4+8y_0U_0^3-16\neq 0$.  But $U_0\neq 0$ follows easily from the fact that \eqref{eq:equilibrium} has a nonzero constant term.  Writing $U_0=(U_0+y_0)-y_0$, we can rewrite \eqref{eq:equilibrium} as a quartic in $U_0+y_0$ with constant term $-\tfrac{1}{2}(y_0^4-8\kappa y_0^2+16)$.  Setting the latter constant term to zero and eliminating $y_0$ between this condition and the branch point condition $B(y_0;\kappa)=0$ (cf.\@ \eqref{eq:branch-points}) yields the condition that $\kappa$ should satisfy either $\kappa=\pm 1$ or $375\kappa^2+3721=0$, neither of which are possible for $\kappa\in\mathbb{R}\setminus\{-1,1\}$.  Hence it also follows that $U_0+y_0\neq 0$.  Finally, eliminating $U_0$ and $y_0$ between  \eqref{eq:equilibrium}, its derivative with respect to $U_0$ (for double equilibria), and $9U_0^4+8y_0U_0^3-16=0$ yields the condition $\kappa^2=1$, so for $\kappa\neq \pm 1$ and $U_0$ a double equilibrium we must also have that $9U_0^4+8y_0U_0^3-16\neq 0$.

This formal analysis suggests that solutions $u(x)$ of Painlev\'e-IV for large $T=|\Theta_0|$ and fixed $\kappa=-\Theta_\infty/T\in \mathbb{R}\setminus\{-1,1\}$ can behave like solutions of the Painlev\'e-I equation when $xT^{-\frac{1}{2}}$ is close to one of the eight branch points satisfying \eqref{eq:branch-points}, provided that also $u\approx T^{\frac{1}{2}}U_0$ for a branching equilibrium $U_0$ in some overlap domain.  The particular solution(s) of Painlev\'e-I that would be relevant is not clear from this formal analysis.  However, in \cite{MasoeroR19} one finds the conjecture that for the gH family of rational solutions one should select a tritronqu\'ee solution of Painlev\'e-I, and in light of the result of \cite{Buckingham18} that any poles or zeros of $u$ should be confined to a region that forms a sector with vertex at a complex branch point $y_0$ and opening angle $\tfrac{2}{5}\pi$ this is a very reasonable hypothesis.  Based on results we will obtain below in Theorem~\ref{thm:OkamotoExterior} one should also expect Painlev\'e-I tritronqu\'ee asymptotics near the branch points on the real and imaginary axes for the gO family of rational solutions.  While the proofs of these tritronqu\'ee convergence results have yet to be given, a similar result has been proven rigorously for rational solutions of the Painlev\'e-II equation in \cite{BuckinghamM15}.  Near the remaining branch points the pole-free sector of the gO rational solutions is smaller, and one can only reasonably anticipate the appearance of a tronqu\'ee solution of Painlev\'e-I.  Such solutions form a one-parameter family containing the tritronqu\'ee solutions as finitely many special cases, so just to formulate a precise conjecture one would need to single out a particular tronqu\'ee solution of Painlev\'e-I.  We note that the formal connection between Painlev\'e-IV and Painlev\'e-I is apparently not a direct link in the ``coalescence cascade" of Painlev\'e equations reported in \cite[\S 32.2(vi)]{DLMF}; in the latter, solutions of Painlev\'e-IV degenerate to solutions of the Painlev\'e-II equation, which in turn can degenerate into solutions of Painlev\'e-I.

\subsection{Results}
\label{sec:Results}
\subsubsection{Related literature}
Our objective is to use the integrable structure of 
\eqref{p4} to analytically prove many of these qualitative observations by computing 
the leading-order asymptotic behavior of the rational Painlev\'e functions 
built from both the generalized Hermite and the generalized Okamoto polynomials.  
Specifically, we use the isomonodromy approach adapted to a Lax pair representation of \eqref{p4} first found in \cite{Jimbo:1981a}.  An outline of how this method leads to Riemann-Hilbert representations of all rational solutions of \eqref{p4} can be found in Section~\ref{sec:intro-Methodology} below.  With these representations in hand, we apply elements of the Deift-Zhou steepest-descent method \cite{DeiftZ93} (in particular, we use the important mechanism of the so-called $g$-function first introduced in \cite{DeiftVZ94}).  This rigorous method of asymptotic analysis, carefully adapted to numerous cases depending on the region of parameter space illustrated in Figure~\ref{fig:ThetasPlane} and the region of the $y_0$-plane under consideration, allows the desired asymptotic formul\ae\ to be proved.  
Families of rational 
solutions to other Painlev\'e equations have recently been analyzed asymptotically, 
including rational solutions of the Painlev\'e-II equation 
\cite{BuckinghamM:2014,BuckinghamM15,BertolaB:2015}, 
rational solutions of the Painlev\'e-II hierarchy \cite{BaloghBB:2016}, and
rational solutions of the Painlev\'e-III equation 
\cite{BothnerMS18,BothnerM:2018}.  All of these works use some sort of Riemann-Hilbert representation and the steepest-descent method. However, in the papers \cite{BertolaB:2015,BaloghBB:2016} the representation used comes from a Hankel determinant identity and the Fokas-Its-Kitaev theory of pseudo-orthogonal polynomials \cite{FokasIK:1991}, while the papers \cite{BothnerMS18,BothnerM:2018,BuckinghamM:2014,BuckinghamM15} follow more the approach described below in Section~\ref{sec:intro-Methodology}.
As for the Painlev\'e-IV equation, the gH family of rational solutions has been studied in \cite{Buckingham18} using the Hankel determinant approach, but so far the gO family has resisted any representation convenient for that method.  The isomonodromy method has been applied to the gO family of rational solutions by Novokshenov and Shchelkonogov \cite{NovokshenovS14}, but only in the special case that $m=0$, i.e., the rational solutions $u_\mathrm{gO}^{[j]}(x;0,n)$ were analyzed for large $n$.  An attempt was made in \cite{NovokshenovS15} to use similar methods for the gH family, but that paper has been shown to contain errors that invalidate its results.  An explicit connection between the Hankel determinant approach and the isomonodromy method (for a suitable Lax pair) was explained for the Painlev\'e-II equation in \cite{MillerS:2017}.  We make a similar connection in this paper for the gH family for Painlev\'e-IV (see \eqref{eq:PIV-Hankel-isomonodromy} at the end of Section~\ref{sec:intro-Methodology}).  From the point of view of isomonodromy theory, it seems that it is the absence of nontrivial Stokes phenomenon in the Lax pair that is correlated with the existence of a Hankel determinant identity suitable for further asymptotic analysis.

We also want to mention here a third approach available to study the roots of the gH and gO polynomials themselves.  It is possible to encode the condition that a gH or gO polynomial vanish at a given point in a kind of eigenvalue condition on a quantum anharmonic oscillator equation in one dimension \cite{MasoeroR18}; see also Remark~\ref{rem:SamePolynomial}.  This method has been used to obtain detailed information about the roots of the gH polynomials \cite{MasoeroR19}, and work is underway to do the same for the gO polynomials \cite{MasoeroR20}.
%, 
%and the rational solutions of the 
%Painlev\'e-IV equation built from generalized Hermite polynomials 
%\cite{Buckingham18,MasoeroR18}.  The large-degree asymptotic behavior 
%of the  has been studied 
%using the isomonodromy method .
%The large-$(m,n)$ asymptotic behavior of the gH polynomials and the related rational solutions of \eqref{p4} has been 
%analyzed using a Riemann-Hilbert representation in terms (pseudo) orthogonal polynomials \cite{Buckingham18} and a connection with the quantum anharmonic oscillator \cite{MasoeroR18,MasoeroR19}.
%Likewise, the rational solutions of \eqref{p4} associated with Okamoto polynomials $Q_{m,0}(x)$ and $Q_{m,1}(x)$ have been studied for large $m$ by Novokshenov and Shchelkonogov \cite{NovokshenovS14} using an approach similar to the one we will take in this work.  
\subsubsection{Equilibrium asymptotics of Painlev\'e-IV rational solutions}
Central to the asymptotic description of rational Painlev\'e-IV solutions are two particular families of Jordan curves that we denote $\partial\HermiteExterior(\kappa)$ and $\partial\OkamotoExterior(\kappa)$ respectively, with the families being parametrized by $\kappa\in\mathbb{R}\setminus\{-1,1\}$.  Given $\kappa$, the curves $\partial\HermiteExterior(\kappa)$ and $\partial\OkamotoExterior(\kappa)$ are finite unions of analytic arcs that can be described as follows.  Let $U_0(y_0;\kappa)$ be a solution of the equilibrium equation \eqref{eq:equilibrium} analytic for $y_0$ in some domain $\mathcal{D}$, and define $E=E(y_0;\kappa)$ in turn by \eqref{eq:E-for-double-roots}.  Then by Remark~\ref{rem:equilibrium-rederive}, for all $y_0\in \mathcal{D}$ the polynomial $P(f)$ in \eqref{eq:elliptic-ODE} has a double root $f=\gamma(y_0;\kappa):=U_0(y_0;\kappa)$ and two simple roots, one of which we denote by $f=\alpha(y_0;\kappa)$.  It follows that the equation 
\eq
\mathrm{Re}\left(\int_{\alpha(y_0;\kappa)}^{\gamma(y_0;\kappa)}\frac{\sqrt{P(f)}}{f}\,\dd f\right)=0
\label{eq:intro-arcs}
\endeq
defines an analytic arc (possibly empty) in $\mathcal{D}$.  On a given domain $\mathcal{D}$ there may be up to four analytic equilibria, and each choice of equilibrium gives different arcs on $\mathcal{D}$.  The arcs forming $\partial\HermiteExterior(\kappa)$ and $\partial\OkamotoExterior(\kappa)$ all arise in this way, although there are some extraneous arcs generated by \eqref{eq:intro-arcs} that are not contained within either $\partial\HermiteExterior(\kappa)$ or $\partial\OkamotoExterior(\kappa)$. 
A more precise description of $\partial\HermiteExterior(\kappa)$ and $\partial\OkamotoExterior(\kappa)$ that makes precise exactly which arcs produced by \eqref{eq:intro-arcs} are needed in each case will be given in Section~\ref{sec:OutsideDomain}, where it is also shown that the arcs arise by conformal mapping from the trajectories of a certain rational quadratic differential (see \eqref{eq:PhiPrimeSquared}).  However, we can formulate the following proposition that describes the most important properties of the Jordan curves $\partial\HermiteExterior(\kappa)$ and $\partial\OkamotoExterior(\kappa)$.  
\begin{proposition}
The families of curves $\partial\HermiteExterior(\kappa)$ and $\partial\OkamotoExterior(\kappa)$ have the following properties.
\begin{enumerate}
\item
For each $\kappa\in \mathbb{R}\setminus\{-1,1\}$:
\begin{enumerate}
\item
$\partial\HermiteExterior(\kappa)$ and $\partial\OkamotoExterior(\kappa)$ are Jordan curves enjoying Schwarz reflection symmetry in both the real and imaginary axes.  
\item
$\partial\HermiteExterior(\kappa)$ consists of four analytic arcs joining in pairs the four branch points $B(y_0;\kappa)=0$ lying in the four open quadrants of the $y_0$-plane, traversed in the direction of increasing $\arg(y_0)$.  
\item 
$\partial\OkamotoExterior(\kappa)$ consists of eight analytic arcs joining in pairs all eight branch points $B(y_0;\kappa)=0$, traversed in the direction of increasing $\arg(y_0)$.
\end{enumerate}
\item
The curves in a given family, $\partial\mathcal{E}_\mathrm{F}(\kappa)$, $\mathrm{F}=\mathrm{gH},\mathrm{gO}$, $\kappa\in\mathbb{R}\setminus\{-1,1\}$, are related to one another by a finite symmetry group of geometric transformations with the following three generators:
\eq
\begin{split}
\partial\mathcal{E}_\mathrm{F}(-\kappa)&=\ii\partial\mathcal{E}_\mathrm{F}(\kappa), \quad\text{(rotation by $\tfrac{1}{2}\pi$)},\\
\partial\mathcal{E}_\mathrm{F}(I^\pm(\kappa))&=\sqrt{\frac{2}{1\pm\kappa}}\partial\mathcal{E}_\mathrm{F}(\kappa),\quad I^\pm(\kappa):=-\frac{\kappa\mp 3}{1\pm\kappa}, \quad\text{(homothetic dilation)},
\end{split}
\label{eq:geometric-symmetries}
\endeq
all defined for $\kappa\in (-1,1)$.  In particular, since $I^+:(-1,1)\to (1,+\infty)$ and $I^-:(-1,1)\to (-\infty,-1)$, the curves $\partial\mathcal{E}_\mathrm{F}(\kappa)$ are determined for all $\kappa\in \mathbb{R}\setminus\{-1,1\}$ as dilations of those curves with $\kappa\in (-1,1)$.  Also, the curves for $\kappa=-3,0,3$ have additional symmetry, being invariant under rotation about the origin by $\tfrac{1}{2}\pi$ radians. 
%\eq
%\OkamotoExterior(\kappa)=\sqrt{\frac{1+\kappa}{2}}\OkamotoExterior\left(\frac{3-\kappa}{1+\kappa}\right)\quad\text{and}\quad \HermiteExterior(\kappa)=\sqrt{\frac{1+\kappa}{2}}\HermiteExterior\left(\frac{3-\kappa}{1+\kappa}\right),\quad\kappa>1,
%\endeq
%\eq
%\OkamotoExterior(\kappa)=\sqrt{\frac{1-\kappa}{2}}\OkamotoExterior\left(-\frac{3+\kappa}{1-\kappa}\right)\quad\text{and}\quad
%\HermiteExterior(\kappa)=\sqrt{\frac{1-\kappa}{2}}\HermiteExterior\left(-\frac{3+\kappa}{1-\kappa}\right),\quad\kappa<-1.
%\endeq
\item
Letting $\HermiteExterior(\kappa)$ and $\OkamotoExterior(\kappa)$ denote the exterior of $\partial\HermiteExterior(\kappa)$ and $\partial\OkamotoExterior(\kappa)$ respectively, 
\begin{enumerate}
\item
for all $\kappa<-1$, the equilibrium $U_{0,\mathrm{gH}}^{[1]}(y_0;\kappa)$ is an analytic function of $y_0\in\HermiteExterior(\kappa)$;
\item
for all $\kappa>1$, the equilibrium $U_{0,\mathrm{gH}}^{[2]}(y_0;\kappa)$ is an analytic function of $y_0\in\HermiteExterior(\kappa)$;
\item
for all $\kappa\in (-1,1)$, the equilibrium $U_{0,\mathrm{gH}}^{[3]}(y_0;\kappa)$ is an analytic function of $y_0\in\HermiteExterior(\kappa)$.
\item
for all $\kappa\in\mathbb{R}\setminus\{-1,1\}$, the equilibrium $U_{0,\mathrm{gO}}(y_0;\kappa)$ is an analytic function of $y_0\in\OkamotoExterior(\kappa)$;
\end{enumerate}
\end{enumerate}
\label{prop:JordanCurves}
\end{proposition}
We give the proof in Section~\ref{sec:ExteriorOkamotoUnivalence}.  Note that by composing the symmetry generators in \eqref{eq:geometric-symmetries} or their inverses several other interesting relations emerge.  For instance the rotation symmetry $\partial\mathcal{E}_\mathrm{F}(-\kappa)=\ii\partial\mathcal{E}_\mathrm{F}(\kappa)$ also holds for $\kappa<-1$ and $\kappa>1$, trivially relating the curves for these ranges of $\kappa$ by inversion of aspect ratio.  There is also a map relating two curves with $\kappa>1$ (or two curves with $\kappa<-1$) by a combination of dilation and rotation, and a pure dilation map relating curves with $\kappa>1$ to curves with $\kappa<-1$.  
In deriving these implied relations it is useful to note that the M\"obius transformations $I^\pm(\kappa)$ are both involutions of the Riemann sphere:  $I^\pm(I^\pm(\kappa))=\kappa$.  
Qualitatively, $\partial\HermiteExterior(\kappa)$ is a curvilinear rectangle symmetric in reflection through both real and imaginary axes, while for the same $\kappa$, $\partial\OkamotoExterior(\kappa)$ replaces each edge of $\partial\HermiteExterior(\kappa)$ with two curvilinear edges having a common vertex in the exterior $\HermiteExterior(\kappa)$ on the axis ray bisecting the original edge.  See Figure~\ref{fig:Domains}.
\begin{figure}[h]
\begin{center}
\includegraphics{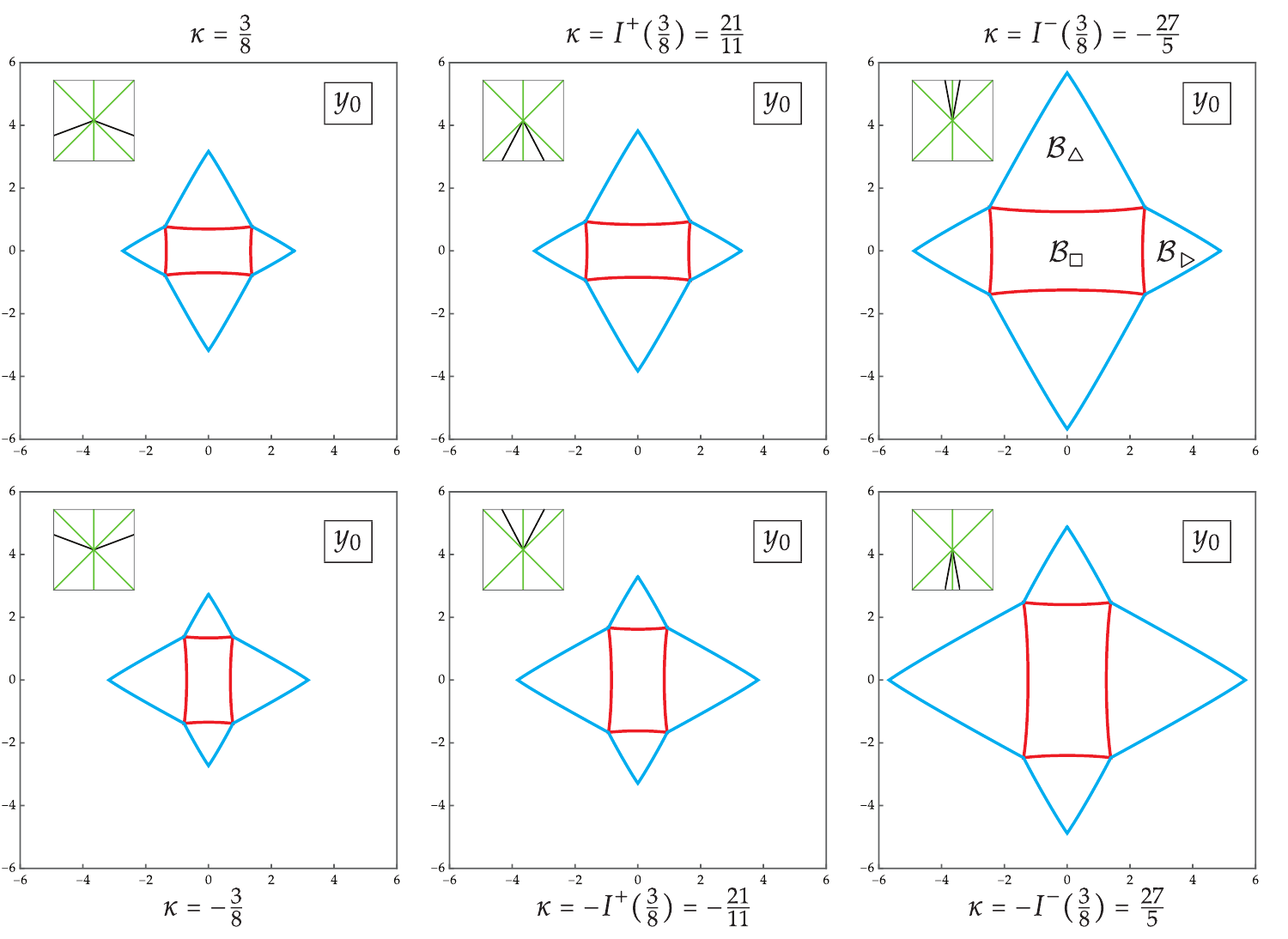}
\end{center}
\caption{The curves $\partial\HermiteExterior$ (red) and $\partial\OkamotoExterior$ (cyan) for six values of $\kappa$ related by the transformations generated by \eqref{eq:geometric-symmetries}, with insets showing the corresponding rays $\Theta_\infty=-\kappa |\Theta_0|$ (black) in the $(\Theta_0,\Theta_\infty)$ plane divided into sectors with green lines.  On the upper right-hand panel are also shown the domains $\rectangle=\rectangle(\kappa)$, $\TR=\TR(\kappa)$, and $\TI=\TI(\kappa)$.}
\label{fig:Domains}
\end{figure}

Our first results assert that the rational solutions are accurately approximated by equilibrium solutions of the formal approximating equation \eqref{eq:Approximating-ODE-Second-Order} (roots of the quartic  \eqref{eq:equilibrium}), provided that $y_0$ lies in the exterior of the relevant Jordan curve.
The following result was first proved in \cite{Buckingham18} using a Hankel determinant identity and techniques from the theory of pseudo-orthogonal polynomials.  We give a new proof in this paper based on 
the isomonodromy method.  For $\rho>0$, set
\eq
\kappa_\infty^{[1]}(\rho):= -1-2\rho^{-1}<-1,\quad \kappa_\infty^{[2]}(\rho):= 1+2\rho>1,\quad\text{and}\quad
\kappa_\infty^{[3]}(\rho):=\frac{1-\rho}{1+\rho}\in (-1,1).
\endeq
\begin{theorem}[Equilibrium asymptotics of gH rationals]
Let $\rho>0$ be a fixed rational aspect ratio, and recall that $\HermiteExterior(\kappa)$ denotes the unbounded exterior of the Jordan curve $\partial\HermiteExterior(\kappa)$.  Then for $j=1,2,3$, as $m,n\to +\infty$ with $n=\rho m$, 
\begin{multline}
u_\mathrm{gH}^{[j]}(x;m,n)=T^\frac{1}{2}\left(\dot{U}+\bo(T^{-1})\right), \quad \dot{U}=U_{0,\mathrm{gH}}^{[j]}\left(y_0;\kappa\right),\quad x=T^\frac{1}{2}y_0,\quad T=|\Theta_{0,\mathrm{gH}}^{[j]}(m,n)|,\\
y_0\in\HermiteExterior(\kappa),\quad\kappa=-\frac{\Theta_{\infty,\mathrm{gH}}^{[j]}(m,n)}{T}\to\kappa_\infty=\kappa_\infty^{[j]}(\rho).
\label{eq:HermiteExterior123}
\end{multline}
%\begin{multline}
%u_\mathrm{gH}^{[1]}(x;m,n)=T^\frac{1}{2}\left(U_{0,\mathrm{gH}}^{[1]}\left(y_0;\kappa\right)+O(T^{-1})\right), \quad x=T^\frac{1}{2}y_0,\quad T=|\Theta_{0,\mathrm{gH}}^{[1]}(m,n)|,\\
%y_0\in\HermiteExterior(\kappa),\quad\kappa=-\frac{\Theta_{\infty,\mathrm{gH}}^{[1]}(m,n)}{T}\to\kappa_\infty=-1-2\rho^{-1}<-1,
%\label{eq:HermiteExterior1}
%\end{multline}
%\begin{multline}
%u_\mathrm{gH}^{[2]}(x;m,n)=T^\frac{1}{2}\left(U_{0,\mathrm{gH}}^{[2]}\left(y_0;\kappa\right)+O(T^{-1})\right),\quad x=T^\frac{1}{2}y_0,\quad T=|\Theta_{0,\mathrm{gH}}^{[2]}(m,n)|,\\
%y_0\in\HermiteExterior(\kappa),\quad
%\kappa=-\frac{\Theta_{\infty,\mathrm{gH}}^{[2]}(m,n)}{T}\to\kappa_\infty=1+2\rho>1,
%\label{eq:HermiteExterior2}
%\end{multline}
%\begin{multline}
%u_\mathrm{gH}^{[3]}(x;m,n)=T^\frac{1}{2}\left(U_{0,\mathrm{gH}}^{[3]}\left(y_0;\kappa\right)+O(T^{-1})\right),\quad x=T^\frac{1}{2}y_0,\quad T=|\Theta_{0,\mathrm{gH}}^{[3]}(m,n)|,\\
%y_0\in\HermiteExterior(\kappa),\quad\kappa=-\frac{\Theta_{\infty,\mathrm{gH}}^{[3]}(m,n)}{T}\to\kappa_\infty=\frac{1-\rho}{1+\rho}\in (-1,1).
%\label{eq:HermiteExterior3}
%\end{multline}
In each case (type $j=1,2,3$) the error term becomes uniform if $\HermiteExterior(\kappa)$ is replaced with a closed (possibly unbounded) subset thereof, and the error tends to zero as $y_0\to\infty$.
\label{thm:HermiteExterior}
\end{theorem}
The pseudo-orthogonal polynomial method has not yet been successfully applied to the gO family of rational solutions, but the isomonodromy approach applies just as well, yielding the following result.
\begin{theorem}[Equilibrium asymptotics of gO rationals]
Let $\rho>0$ be a fixed rational aspect ratio, and recall that $\OkamotoExterior(\kappa)$ denotes the unbounded exterior of the Jordan curve $\partial\OkamotoExterior(\kappa)$.  Then for $j=1,2,3$, as $m,n\to \pm\infty$ with $n= \rho m$, 
\begin{multline}
u_\mathrm{gO}^{[j]}(x;m,n)=T^\frac{1}{2}\left(\dot{U}+\bo(T^{-1})\right),\quad\dot{U}=U_{0,\mathrm{gO}}\left(y_0;\kappa\right),\quad x=T^\frac{1}{2}y_0,\quad T=|\Theta_{0,\mathrm{gO}}^{[j]}(m,n)|,\\
y_0\in\OkamotoExterior(\kappa),\quad
\kappa=-\frac{\Theta_{\infty,\mathrm{gO}}^{[j]}(m,n)}{T}\to\kappa_\infty=\pm\kappa_\infty^{[j]}(\rho).
\label{eq:OkamotoExterior123}
\end{multline}
%\begin{multline}
%u_\mathrm{gO}^{[1]}(x;m,n)=T^\frac{1}{2}\left(U_{0,\mathrm{gO}}\left(y_0;\kappa\right)+O(T^{-1})\right),\quad x=T^\frac{1}{2}y_0,\quad T=|\Theta_{0,\mathrm{gO}}^{[1]}(m,n)|,\\
%y_0\in\OkamotoExterior(\kappa),\quad
%\kappa=-\frac{\Theta_{\infty,\mathrm{gO}}^{[1]}(m,n)}{T}\to\kappa_\infty=-\mathrm{sgn}(n)(1+2\rho^{-1})\in (-\infty,-1)\cup(1,+\infty),
%\label{eq:OkamotoExterior1}
%\end{multline}
%\begin{multline}
%u_\mathrm{gO}^{[2]}(x;m,n)=T^\frac{1}{2}\left(U_{0,\mathrm{gO}}\left(y_0;\kappa\right)+O(T^{-1})\right),\quad x=T^\frac{1}{2}y_0,\quad T=|\Theta_{0,\mathrm{gO}}^{[2]}(m,n)|,\\
%y_0\in\OkamotoExterior(\kappa),\quad \kappa=-\frac{\Theta_{\infty,\mathrm{gO}}^{[2]}(m,n)}{T}\to\kappa_\infty=\mathrm{sgn}(m)(1+2\rho)\in (-\infty,-1)\cup(1,+\infty),
%\label{eq:OkamotoExterior2}
%\end{multline}
%\begin{multline}
%u_\mathrm{gO}^{[3]}(x;m,n)=T^\frac{1}{2}\left(U_{0,\mathrm{gO}}\left(y_0;\kappa\right)+O(T^{-1})\right),\quad x=T^\frac{1}{2}y_0,\quad T=|\Theta_{0,\mathrm{gO}}^{[3]}|,\\
%y_0\in\OkamotoExterior(\kappa),\quad \kappa=-\frac{\Theta_{\infty,\mathrm{gO}}^{[3]}(m,n)}{T}\to\kappa_\infty=\mathrm{sgn}(m+n)\frac{1-\rho}{1+\rho}\in (-1,1).
%\label{eq:OkamotoExterior3}
%\end{multline}
In each case (type $j=1,2,3$) the error term becomes uniform if $\OkamotoExterior(\kappa)$ is replaced with a closed (possibly unbounded) subset thereof, and the error tends to zero as $y_0\to\infty$.
\label{thm:OkamotoExterior}
\end{theorem}
An elementary corollary of these results is the following, in which $\mathrm{F}=\mathrm{gH}$ or $\mathrm{F}=\mathrm{gO}$, and type $j=1,2,3$ is arbitrary.
\begin{corollary}[Pole- and zero-free regions for Painlev\'e-IV rational solutions]
Let $C\subset\mathcal{E}_\mathrm{F}(\kappa_\infty)$ be closed (note that $\kappa_\infty$ is a different function of aspect ratio $\rho$ for different types $j=1,2,3$).  Then $u_\mathrm{F}^{[j]}(x;m,n)$ has no poles or zeros for $y_0\in C$ if $m,n$ are sufficiently large and $n=\rho m$.
\label{cor:NoPolesZeros}
\end{corollary}
\begin{proof}
For the gO family, the uniform convergence on $C$ to an analytic equilibrium guaranteed by Theorem~\ref{thm:OkamotoExterior} proves the absence of poles.  Then the argument principle implies the absence of zeros, since the equilibrium $U_{0,\mathrm{gO}}(y_0;\kappa)$ is analytic and bounded away from zero on $F$ according to \eqref{eq:equilibrium} and the asymptotic $U_{0,\mathrm{gO}}(y_0;\kappa)\sim -\tfrac{2}{3}y_0$ as $y_0\to\infty$.

For the gH family, the proof is similar, except that the equilibria $U_{0,\mathrm{gH}}^{[1]}(y_0;\kappa)$ and $U_{0,\mathrm{gH}}^{[2]}(y_0;\kappa)$ vanish as $y_0\to\infty$, being proportional to $y_0^{-1}$.  In these cases, however, the fact that the error term vanishes as $y_0\to\infty$ shows that the simple zero of the equilibrium at infinity cannot be perturbed into the finite $y_0$-plane should $C$ be taken to be unbounded.
\end{proof}

The fact that the error terms in Theorems~\ref{thm:HermiteExterior} and \ref{thm:OkamotoExterior} vanish as $y_0\to\infty$ follows by comparing the known large-$x$ asymptotic behavior of the rational solution (see \eqref{eq:gHsolutionsLarge-x} and \eqref{eq:gOsolutionsLarge-x}) with the large-$y_0$ behavior of the leading terms (see \eqref{eq:equilibria-near-infinity}).  Note that combining the symmetries \eqref{eq:symmetry-1-2} and \eqref{eq:equilibrium-symmetry} with the fact following from Proposition~\ref{prop:JordanCurves} that 
$\HermiteExterior(-\kappa)=\ii\HermiteExterior(\kappa)$ and $\OkamotoExterior(-\kappa)=\ii\OkamotoExterior(\kappa)$ holds for all $\kappa\in\mathbb{R}\setminus\{-1,1\}$, formul\ae\ \eqref{eq:HermiteExterior123} and \eqref{eq:OkamotoExterior123} for $j=2$ follow from the same formul\ae\ with $j=1$.  For rational functions of types $j=1,3$, the asymptotic formul\ae\ in Theorems~\ref{thm:HermiteExterior} and \ref{thm:OkamotoExterior} are proved in Sections~\ref{sec:Exterior}--\ref{sec:Exterior-gH}.  The domains of validity of these formul\ae\ are obtained in Section~\ref{sec:OutsideDomain} by determining exactly which arcs generated by \eqref{eq:intro-arcs} are relevant.  There the domains $\HermiteExterior(\kappa)$ and $\OkamotoExterior(\kappa)$ are precisely specified in 
Definition~\ref{def:ExteriorDomains}, and uniformity of convergence is discussed in Section~\ref{sec:Exterior-Uniformity}.

\subsubsection{Nonequilibrium asymptotics of Painlev\'e-IV rational solutions}
\label{sec:intro-results-nonequilibrium}
The Jordan curve $\partial\HermiteExterior(\kappa)$ lies in the closure of the interior of $\partial\OkamotoExterior(\kappa)$, and it intersects the Jordan curve $\partial\OkamotoExterior(\kappa)$ only at the four branch points (solutions $y_0$ of $B(y_0;\kappa)=0$; see \eqref{eq:branch-points}) that do not lie on the real or imaginary axes.  The interior of $\partial\HermiteExterior(\kappa)$ is a domain that we call $\rectangle(\kappa)$, and the intersection of $\HermiteExterior$ with the interior of $\partial\OkamotoExterior(\kappa)$ is a disjoint union of four domains, each of which intersects just one of the four coordinate axes in the $y_0$-plane.  We call the domain intersecting the positive real (resp.\@ imaginary) axis $\TR(\kappa)$ (resp.\@ $\TI(\kappa)$).  The domain $\rectangle(\kappa)$ is a curvilinear rectangle, while the domains $\TR(\kappa)$ and $\TI(\kappa)$ are curvilinear triangles.  Note that although Proposition~\ref{prop:triangles} asserts that the vertices of $\TR(\kappa)$ and $\TI(\kappa)$ are those of exact equilateral triangles, their edges are analytic arcs that neither are straight-line segments, nor are symmetric under rotation about the center by $\tfrac{2}{3}\pi$ radians.  See the upper right-hand panel of Figure~\ref{fig:Domains}.  The remaining two domains are then $-\TR(\kappa)$ and $-\TI(\kappa)$ (reflections through the origin).  We now describe the gH rational solutions of Painlev\'e-IV for values of $x$ corresponding to $y_0\in\rectangle(\kappa)$ and the gO rational solutions for $x$ corresponding to $y_0\in\rectangle(\kappa)\cup\TR(\kappa)\cup\TI(\kappa)$.  In light of Theorems~\ref{thm:HermiteExterior}--\ref{thm:OkamotoExterior} and the odd and Schwarz-reflection symmetries of every rational solution $u(x)$ described in Proposition~\ref{prop:FirstQuadrant}, this exhausts the complex $y_0$-plane except for a neighborhood of the curve $\partial\HermiteExterior$ and (for the gO family only) of the curve $\partial\OkamotoExterior$.  

Recall the quartic polynomial $P(f)=P(f;y_0,\kappa,E)$ defined in \eqref{eq:elliptic-ODE}.  Given $\kappa\in\mathbb{R}\setminus\{-1,1\}$ and $y_0\in\rectangle(\kappa)\cup\TR(\kappa)\cup\TI(\kappa)$, there is a specific value of $E\in\mathbb{C}$ such that the conditions (compare with \eqref{eq:intro-arcs})
\eq
\mathrm{Re}\left(\oint_\mathfrak{a}\frac{\sqrt{P(f)}}{f}\,\dd f\right)=\mathrm{Re}\left(\oint_\mathfrak{b}\frac{\sqrt{P(f)}}{f}\,\dd f\right)=0
\label{eq:intro-Boutroux}
\endeq
both hold.  The conditions \eqref{eq:intro-Boutroux} also appear in the Masoero-Roffelsen theory of zeros of gH polynomials, where they serve to define a mapping $\mathcal{S}$ (see \cite[Eqn.\@ (9)]{MasoeroR19}) used to localized the zeros.  Note that these integrals are independent of path on the Riemann surface of $w^2=P(z)$ because the differential in the integrand, while singular over $z=0,\infty$ has purely real residues.  Taken together, they are also independent of specific choice of basis of cycles $\mathfrak{a}$ and $\mathfrak{b}$.  The specific value we use is denoted $E=E(y_0;\kappa)$ and is properly defined in Definition~\ref{def:E} below.
The most important properties of $E(y_0;\kappa)$ are the following.
\begin{proposition}
For each $\kappa\in \mathbb{R}\setminus\{-1,1\}$ and $y_0\in\rectangle(\kappa)\cup\TR(\kappa)\cup\TI(\kappa)$, the quartic $P(f;y_0,\kappa,E(y_0;\kappa))$ satisfies the conditions \eqref{eq:intro-Boutroux}.  If $\kappa$ is fixed, the function $y_0\mapsto E$ is smooth but not analytic on each component of its domain, and it extends continuously to $\partial\OkamotoExterior(\kappa)$.    The functions $y_0\mapsto E$ are related for different $\kappa$ by the following symmetries:
\eq
E(y_0;I^\pm(\kappa))=\left(\frac{2}{1\pm\kappa}\right)^\frac{3}{2}\left[E\left(\sqrt{\tfrac{1}{2}(1\pm\kappa)}y_0;\kappa\right)-4(\kappa\mp 1)\sqrt{\tfrac{1}{2}(1\pm\kappa)}y_0\right],\quad I^\pm(\kappa):=-\frac{\kappa\mp 3}{1\pm\kappa},
\label{eq:E-twist-symmetry}
\endeq
\eq
E(y_0;-\kappa)=\ii E(\ii y_0;\kappa).
\label{eq:E-flip-symmetry}
\endeq
In particular, the latter symmetry implies that $E(-y_0;\kappa)=-E(y_0;\kappa)$, i.e., $E$ is an odd function of $y_0$ for each $\kappa\in\mathbb{R}\setminus\{-1,1\}$.
\label{prop:E}
\end{proposition}

Setting $E=E(y_0;\kappa)$ in the differential equation \eqref{eq:elliptic-ODE}, let $f(\zeta)=f(\zeta;y_0,\kappa)$ denote the unique solution of this equation satisfying $f(0)=0$ and $f'(0)=4$, and let $\zeta_\infty^\pm$ denote any two poles of $f(\zeta)$ of residue $\pm 1$.  
%For technical reasons related to our method of proof, we need to define lattices $\mathcal{M}^{[j]}$, $j=1,2,3$ depending on $y_0$ and $\kappa$ by 
%\eq
%\mathcal{M}^{[1]}:=\{\zeta\in\mathbb{C}:\zeta=\zeta_\infty^-\pmod{Z_\mathfrak{a}\mathbb{Z}+Z_\mathfrak{b}\mathbb{Z}}\}\quad\text{and}\quad
%\mathcal{M}^{[2]}=\mathcal{M}^{[3]}:=\{\zeta\in\mathbb{C}:\zeta=\zeta_\infty^+\pmod{Z_\mathfrak{a}\mathbb{Z}+Z_\mathfrak{b}\mathbb{Z}}\}.
%\label{eq:MalgrangeLattice}
%\endeq
The next results concern the approximation of the rational Painlev\'e-IV solutions by suitable phase shifts of the elliptic function $f$.  Define an exponent by
\eq
\chi(\zeta):=\begin{cases} 1,&\quad |f(\zeta)|\le 1\\
-1,&\quad |f(\zeta)|>1.
\end{cases}
\endeq
The purpose of $\chi(\zeta)$ is to allow a streamlined asymptotic description of rational functions near both zeros and poles.
%; see Remark~\ref{rem:Malgrange} for further details.
\begin{theorem}[Elliptic asymptotics of gH rationals]
Let $\rho>0$ be a fixed rational aspect ratio.  Then for $j=1,2,3$ there exists a family of smooth but not analytic maps $y_0\mapsto\zeta_0=\zeta_{0,\mathrm{gH}}^{[j]}(y_0;m,n)$, $\rectangle(\kappa)\to\mathbb{C}$, 
%defined for $y_0\in\rectangle(\kappa)$ 
such that as $m,n\to +\infty$ with $n=\rho m$, 
\begin{multline}
u_\mathrm{gH}^{[j]}(x;m,n)^{\chi}=T^{\frac{1}{2}\chi}\left(\dot{U}^\chi+\bo(T^{-1})\right),\quad \dot{U}=f(\zeta-\zeta_0),\quad x=T^\frac{1}{2}y_0+T^{-\frac{1}{2}}\zeta,\quad T=|\Theta_{0,\mathrm{gH}}^{[j]}(m,n)|,\\
y_0\in\rectangle(\kappa),\quad
\kappa=-\frac{\Theta_{\infty,\mathrm{gH}}^{[j]}(m,n)}{T}\to\kappa_\infty=\kappa_\infty^{[j]}(\rho),
\end{multline}
where $\chi=\chi(\zeta-\zeta_0)$, and where the error term is uniform for $y_0$ in a compact subset of $\rectangle(\kappa)$ and for $\zeta$ bounded. 
\label{thm:Hermite-elliptic}
\end{theorem}
\begin{theorem}[Elliptic asymptotics of gO rationals]
Let $\rho>0$ be a fixed rational aspect ratio, and let $\mathcal{B}(\kappa)$ denote either $\rectangle(\kappa)$, $\TR(\kappa)$, or $\TI(\kappa)$.  Then for $j=1,2,3$ there exists a family of smooth but not analytic maps $y_0\mapsto\zeta_0=\zeta_{0,\mathrm{gO}}^{[j]}(y_0;m,n)$, $\mathcal{B}(\kappa)\to\mathbb{C}$, 
%defined for $y_0\in\mathcal{B}(\kappa)$ 
such that as $m,n\to\pm\infty$ with $n=\rho m$,
\begin{multline}
u_\mathrm{gO}^{[j]}(x;m,n)^{\chi}=T^{\frac{1}{2}\chi}\left(\dot{U}^{\chi}+\bo(T^{-1})\right),\quad \dot{U}=f(\zeta-\zeta_0),\quad x=T^\frac{1}{2}y_0+T^{-\frac{1}{2}}\zeta,\quad T=|\Theta_{0,\mathrm{gO}}^{[j]}(m,n)|,\\
y_0\in\mathcal{B}(\kappa),\quad
\kappa=-\frac{\Theta_{\infty,\mathrm{gO}}^{[j]}(m,n)}{T}\to\kappa_\infty=\pm\kappa_\infty^{[j]}(\rho),
\label{eq:Okamoto-elliptic}
\end{multline}
where $\chi=\chi(\zeta-\zeta_0)$, and where the error term is uniform for $y_0$ in a compact subset of $\mathcal{B}(\kappa)$ and for $\zeta$ bounded.
\label{thm:Okamoto-elliptic}
\end{theorem}

These results therefore assert the existence of the phase shift $\zeta_0$ as a function of $y_0$ and large parameters $(m,n)$ for which the elliptic approximation formally described in Section~\ref{sec:nonequilibrium}, but with a specific value $E=E(y_0;\kappa)$ of the integration constant, is accurate to the indicated order.  We will not give formul\ae\ for the phase shift (but see Section~\ref{sec:nonequilibrium} for the special case of $y_0=0$), because in our proofs of these results we actually use a different representation of $f(\zeta-\zeta_0)$ in terms of theta functions.  This representation has the advantage of isolating the two different lattices of poles and zeros of $f(\zeta-\zeta_0)$, allowing for comparison with the roots of the four different polynomial factors in each exact rational solution according to the representations shown in the right-most columns of Tables~\ref{tab:gH} and \ref{tab:gO}.
%\eqref{eq:gH-1-solution}, \eqref{eq:gH-2-solution}, \eqref{eq:gH-3-solution}, \eqref{eq:gO-1m-solution}, \eqref{eq:gO-2m-solution}, and \eqref{eq:gO-3p-solution}. 
Specifically, we prove that, for a suitable canonical homology basis underlying the periods $Z_\mathfrak{a}$ and $Z_\mathfrak{b}$ defined in \eqref{eq:elliptic-periods}, 
the leading term $f(\zeta-\zeta_0)$ in Theorems~\ref{thm:Hermite-elliptic} and \ref{thm:Okamoto-elliptic} can be written as
\eq
f(\zeta-\zeta_0)=\psi(y_0)\frac{\Theta(\frac{2\pi\ii}{Z_\mathfrak{a}}\zeta+\ii\xi(y_0;m,n)+\ii\mathfrak{z}_1(y_0)+K)\Theta(\frac{2\pi\ii}{Z_\mathfrak{a}}\zeta+\ii\xi(y_0;m,n)+\ii\mathfrak{z}_2(y_0)+K)}{\Theta(\frac{2\pi\ii}{Z_\mathfrak{a}}\zeta+\ii\xi(y_0;m,n)+\ii\mathfrak{p}_1(y_0)+K)\Theta(\frac{2\pi\ii}{Z_\mathfrak{a}}\zeta+\ii\xi(y_0;m,n)+\ii\mathfrak{p}_2(y_0)+K)},
\label{eq:intro-elliptic-theta}
\endeq
%\begin{multline}
%u^{[3]}_\mathrm{F}(x;m,n)=\sqrt{|\Theta_0|}\left[\psi(y_0)\frac{\Theta\left(\frac{2\pi\ii}{Z_\mathfrak{a}}\zeta + \phi^0_1(y_0;m,n) + K\right)
%\Theta\left(\frac{2\pi\ii}{Z_\mathfrak{a}}\zeta+\phi^0_2(y_0;m,n)+K\right)}
%{\Theta\left(\frac{2\pi\ii}{Z_\mathfrak{a}}\zeta + \phi^\infty_1(y_0;m,n) +K\right)
%\Theta\left(\frac{2\pi\ii}{Z_\mathfrak{a}}\zeta+\phi^\infty_2(y_0;m,n)+ K\right)}+O(|\Theta_0|^{-1})\right],\\
%x=|\Theta_0|^\frac{1}{2}y_0+|\Theta_0|^{-\frac{1}{2}}\zeta,
%\end{multline}
%where $\Theta_0=\Theta_{0,\mathrm{F}}^{[3]}(m,n)$, 
where only the common phase $\xi(y_0;m,n)$ contains terms proportional to the large parameters $(m,n)$,  $K:=\ii\pi(1+Z_\mathfrak{b}/Z_\mathfrak{a})$, and $\Theta(z)$ is the Riemann theta function for the homology basis $(\mathfrak{a},\mathfrak{b})$ of the elliptic curve $w^2=P(z)$.  The latter is an entire function of $z$ with simple zeros only at the points $z=K+2\pi\ii k + 2\pi\ii\ell Z_\mathfrak{b}/Z_\mathfrak{a}$ for $(k,\ell)\in\mathbb{Z}\times\mathbb{Z}$.  The bounded and nonvanishing factor $\psi(y_0)$ and the phases $\xi(y_0;m,n)$, $\mathfrak{z}_{1,2}(y_0)$, and $\mathfrak{p}_{1,2}(y_0)$ are different for each family and type of rational solution and, in the case of the gO family, for each of the three regions $\rectangle(\kappa)$, $\TR(\kappa)$, and $\TI(\kappa)$.  While the various ingredients in the formula \eqref{eq:intro-elliptic-theta} appear naturally as part of the proof of accuracy, for the reader's convenience we summarize in Appendix~\ref{app:Effective} how these ingredients are effectively computed for each family (gH or gO), type ($j=1,2,3$), and region $\mathcal{B}(\kappa)$.

A characteristic feature of the approximation formul\ae\ in Theorems~\ref{thm:Hermite-elliptic} and \ref{thm:Okamoto-elliptic} is that while the rational function approximated depends on only one variable say $y=T^{-\frac{1}{2}}x=y_0+T^{-1}\zeta$, the approximation $T^{\frac{1}{2}}f(\zeta-\zeta_0)$ essentially involves $y_0$ and $\zeta$ as independent variables, and while it is certainly meromorphic in $\zeta$, it has a nonzero $\overline{\partial}$ derivative with respect to $y_0$.  This means that for a given value of $y=T^{-\frac{1}{2}}x$, the approximation formula in each of these theorems offers a one-parameter family of different approximations of the same rational function.  Two particularly useful ways to make use of this freedom are (i) to fix $\zeta$ and vary $y_0$ within one of the regions 
$\rectangle(\kappa)$ or $\TR(\kappa)$ or $\TI(\kappa)$ or (ii) to fix $y_0$ and instead allow $\zeta$ to vary in a bounded set.  The approach (i) yields an approximation that is uniformly accurate over a given compact set in $y_0$ that corresponds to a large region of size $T^\frac{1}{2}$ in the $x$-plane, but the approximation fails to be meromorphic in $y_0$ (its $\overline{\partial}$ derivative is of course exactly cancelled by that of the error term, although that term is not known in any detail).  On the other hand, the approach (ii) yields an approximation that is an exact elliptic function, hence meromorphic, but the approximation is only accurate for bounded $\zeta$, which corresponds to a small region of size $T^{-\frac{1}{2}}$ in the $x$-plane.  Putting the two approaches together, one can think of $(y_0,\zeta)$ as coordinates on the complex tangent bundle over $\rectangle(\kappa)$, $\TR(\kappa)$, or $\TI(\kappa)$. These two interpretations of the asymptotic formul\ae\ in Theorems~\ref{thm:Hermite-elliptic} and \ref{thm:Okamoto-elliptic} are consistent with the simple idea that in a given region the rational Painlev\'e-IV functions are approximated by an elliptic function whose modulus and phase shift vary slowly on scales that are large compared to the periods $Z_\mathfrak{a}$ and $Z_\mathfrak{b}$.

Using the approach (i) allows one to combine the equilibrium asymptotics of Theorems~\ref{thm:HermiteExterior}--\ref{thm:OkamotoExterior} with the elliptic asymptotics of Theorems~\ref{thm:Hermite-elliptic}--\ref{thm:Okamoto-elliptic} and obtain an approximation $\dot{U}$ of $U=T^{-\frac{1}{2}}u_\mathrm{F}^{[j]}(T^\frac{1}{2}y_0;m,n)$, defined piecewise in the complex $y_0$-plane, the accuracy of which is guaranteed for all $y_0\in\mathbb{C}$ except for values near the Jordan curves $\partial\HermiteExterior(\kappa)$ and $\partial\OkamotoExterior(\kappa)$.  These approximations are remarkably accurate even for $(m,n)$ not too large, as can be seen in Figures~\ref{fig:gH-axes}, \ref{fig:gOpp-axes}, and \ref{fig:gOnn-axes}.
\begin{figure}[h]
\begin{center}
\begin{subfigure}{1.5 in}
\begin{center}
\includegraphics[height=1.0 in]{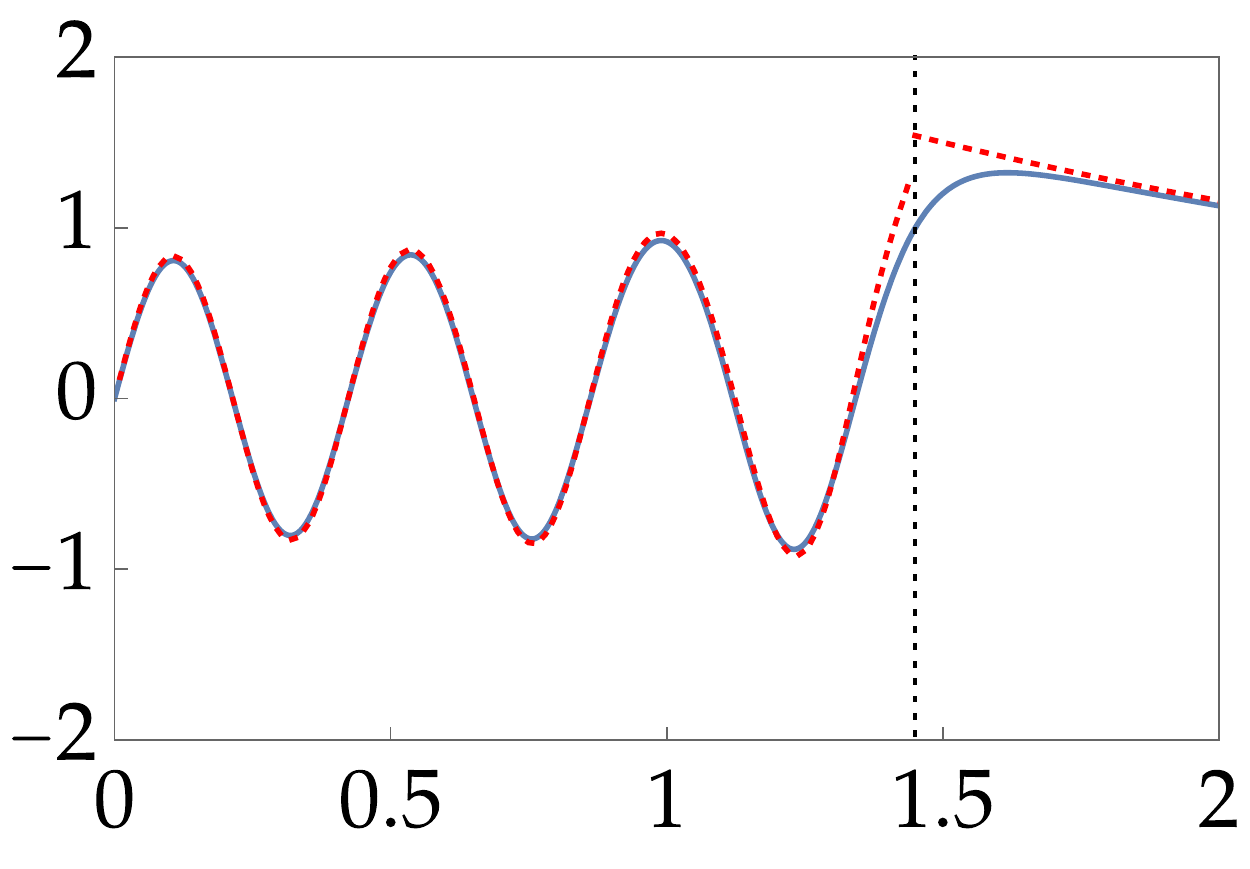}\\
\includegraphics[height=1.0 in]{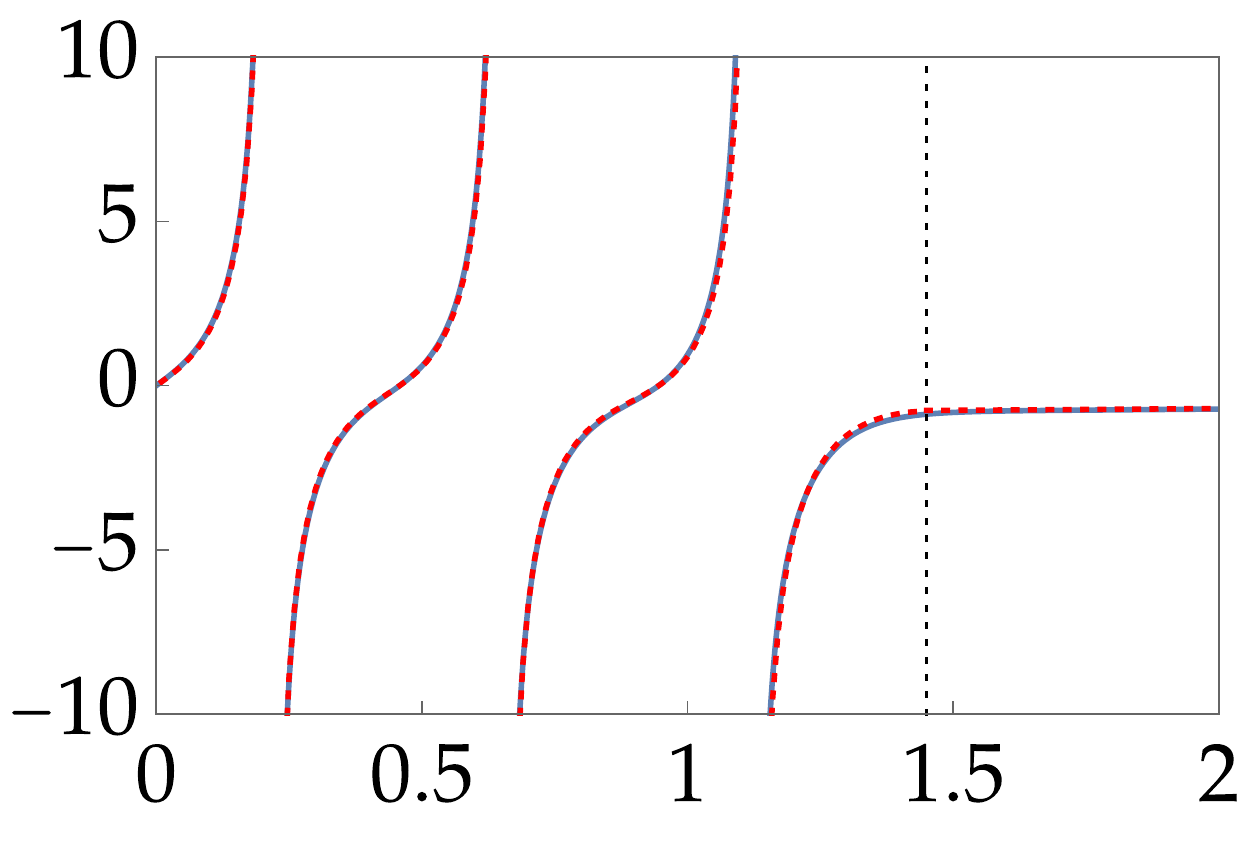}
\end{center}
\caption{$U=T^{-\frac{1}{2}}u^{[1]}_\mathrm{gH}(x;6,6)$; \\$\rho=1$; $\kappa=-\tfrac{10}{3}$.}
\end{subfigure}%
\begin{subfigure}{1.5 in}
\begin{center}
\includegraphics[height=1.0 in]{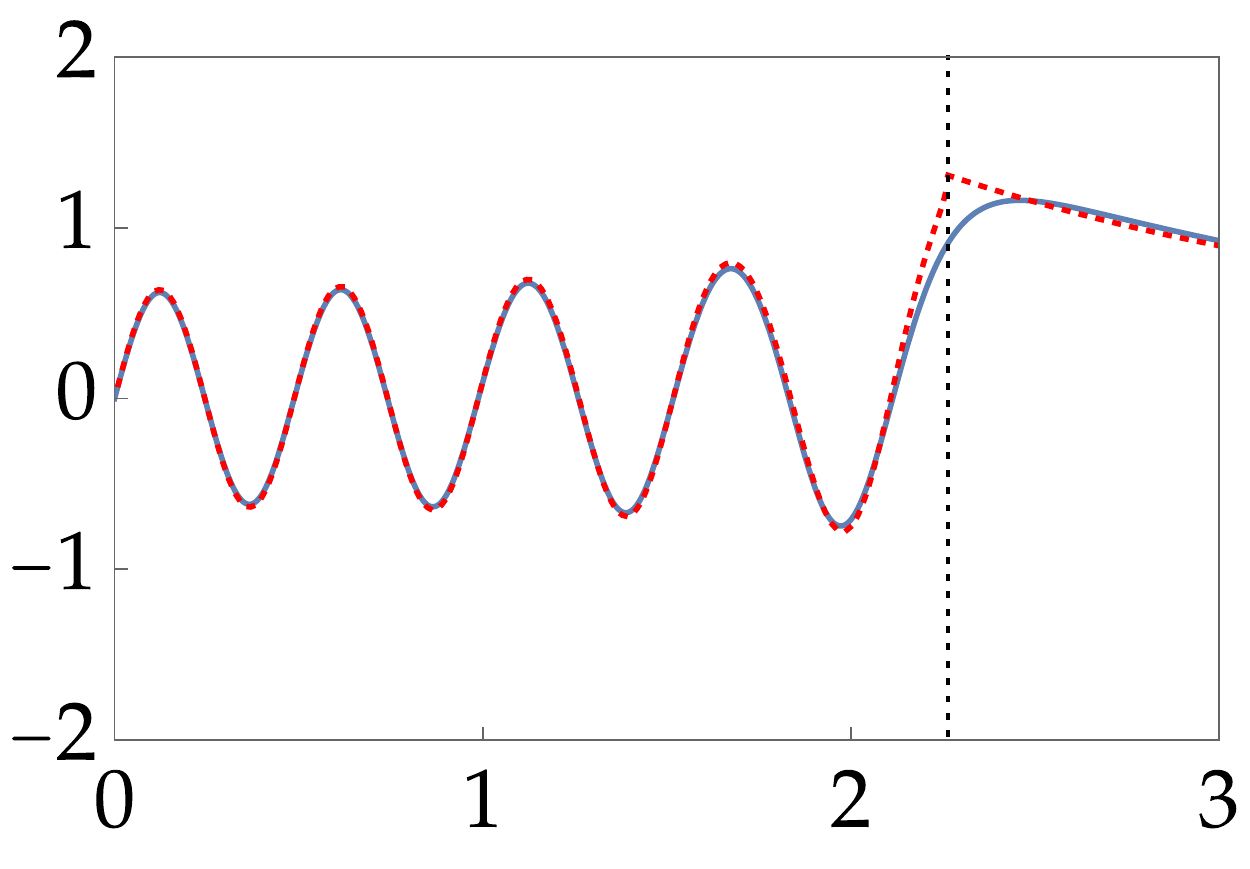}\\
\includegraphics[height=1.0 in]{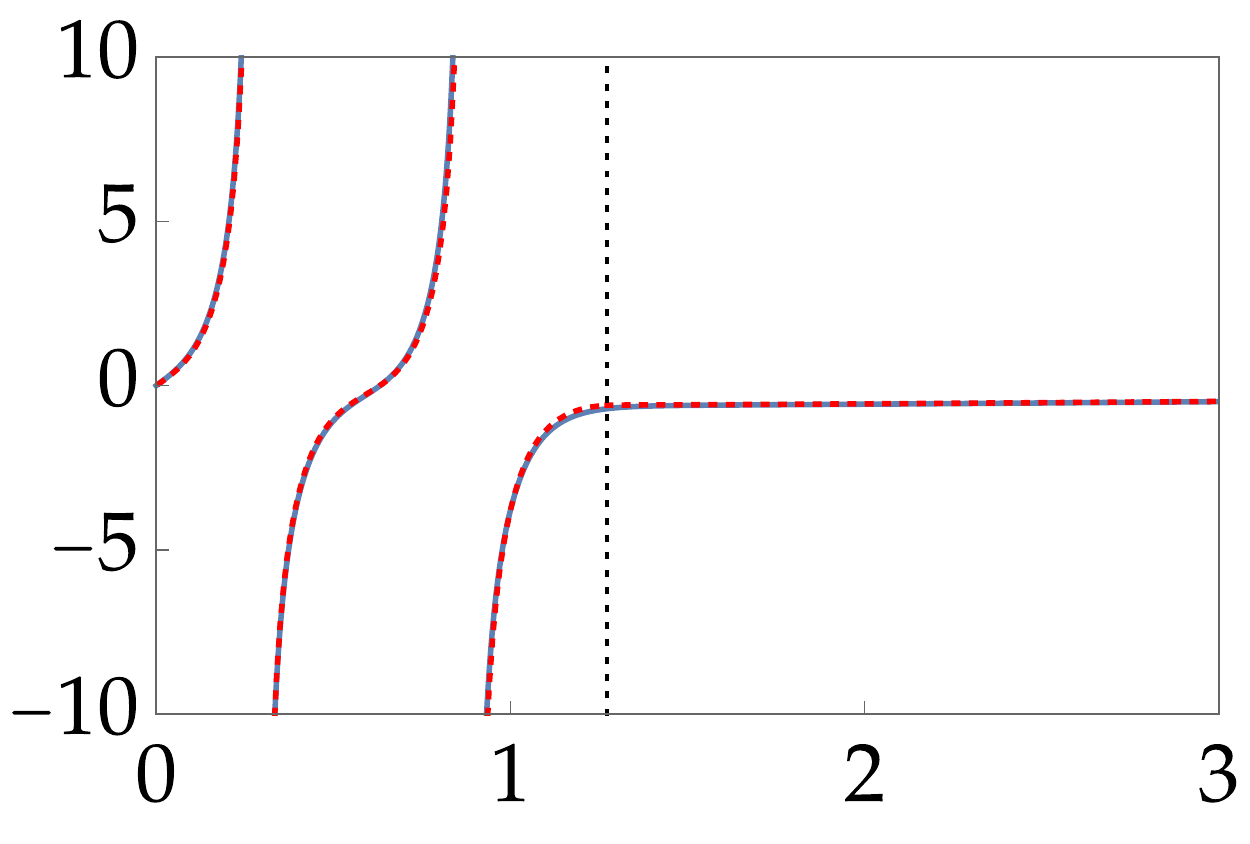}
\end{center}
\caption{$U=T^{-\frac{1}{2}}u^{[1]}_\mathrm{gH}(x;8,4)$; \\$\rho=\tfrac{1}{2}$; $\kappa=-\tfrac{11}{2}$.}
\end{subfigure}%
\begin{subfigure}{1.5 in}
\begin{center}
\includegraphics[height=1.0 in]{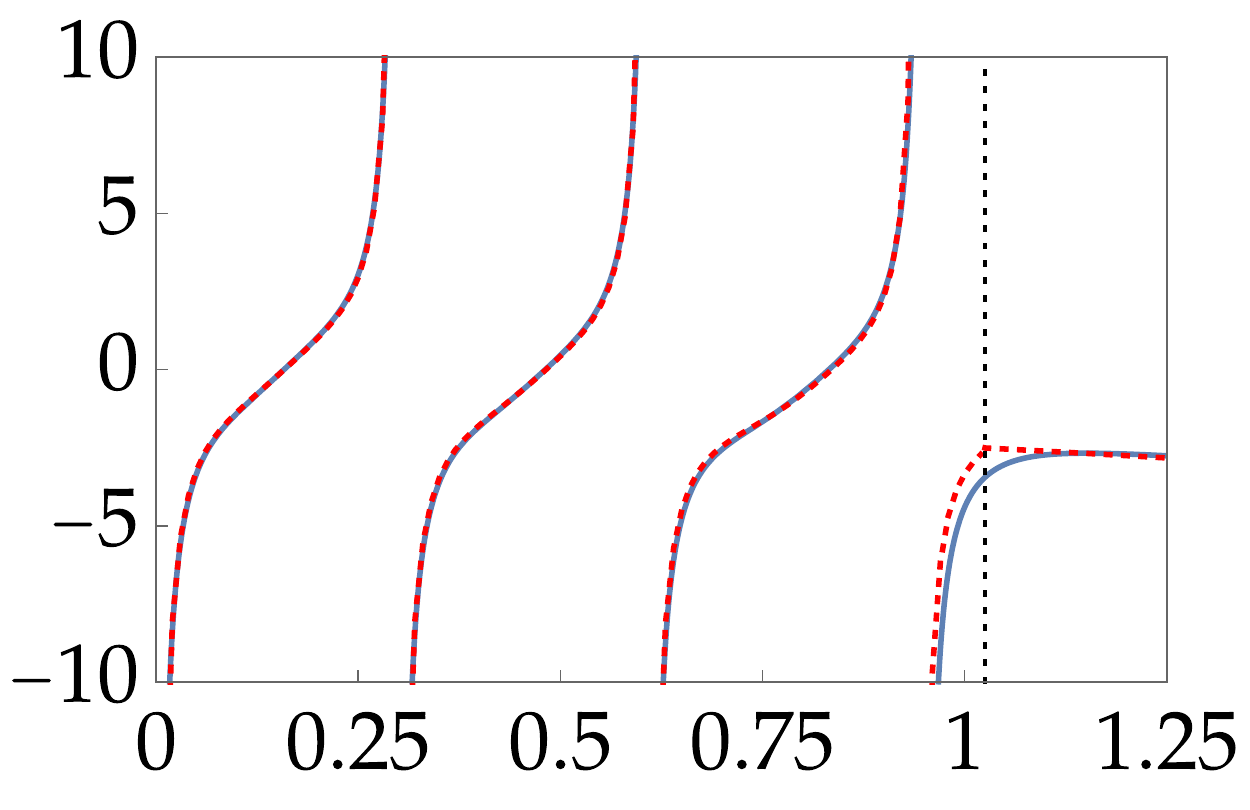}\\
\includegraphics[height=1.0 in]{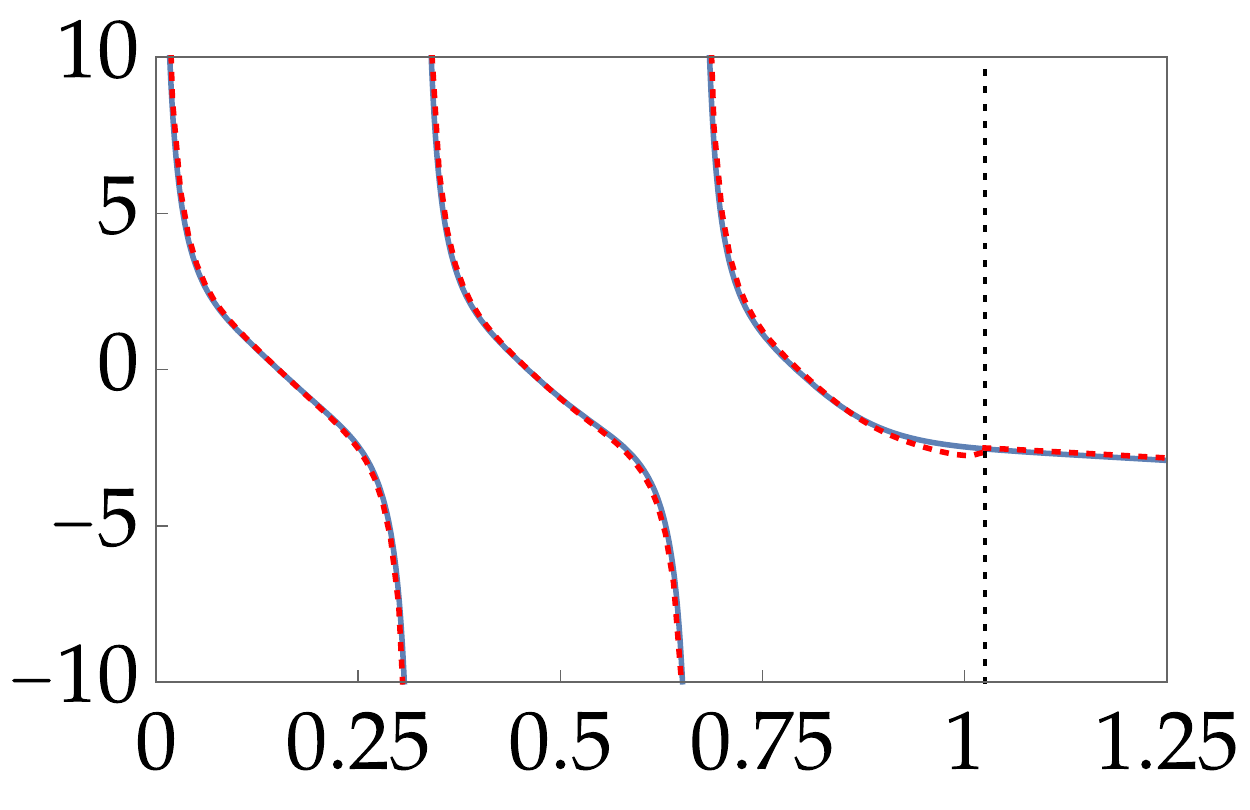}
\end{center}
\caption{$U=T^{-\frac{1}{2}}u^{[3]}_\mathrm{gH}(x;6,5)$; \\$\rho=\tfrac{5}{6}$; $\kappa=0$.}
\end{subfigure}%
\begin{subfigure}{1.5 in}
\begin{center}
\includegraphics[height=1.0 in]{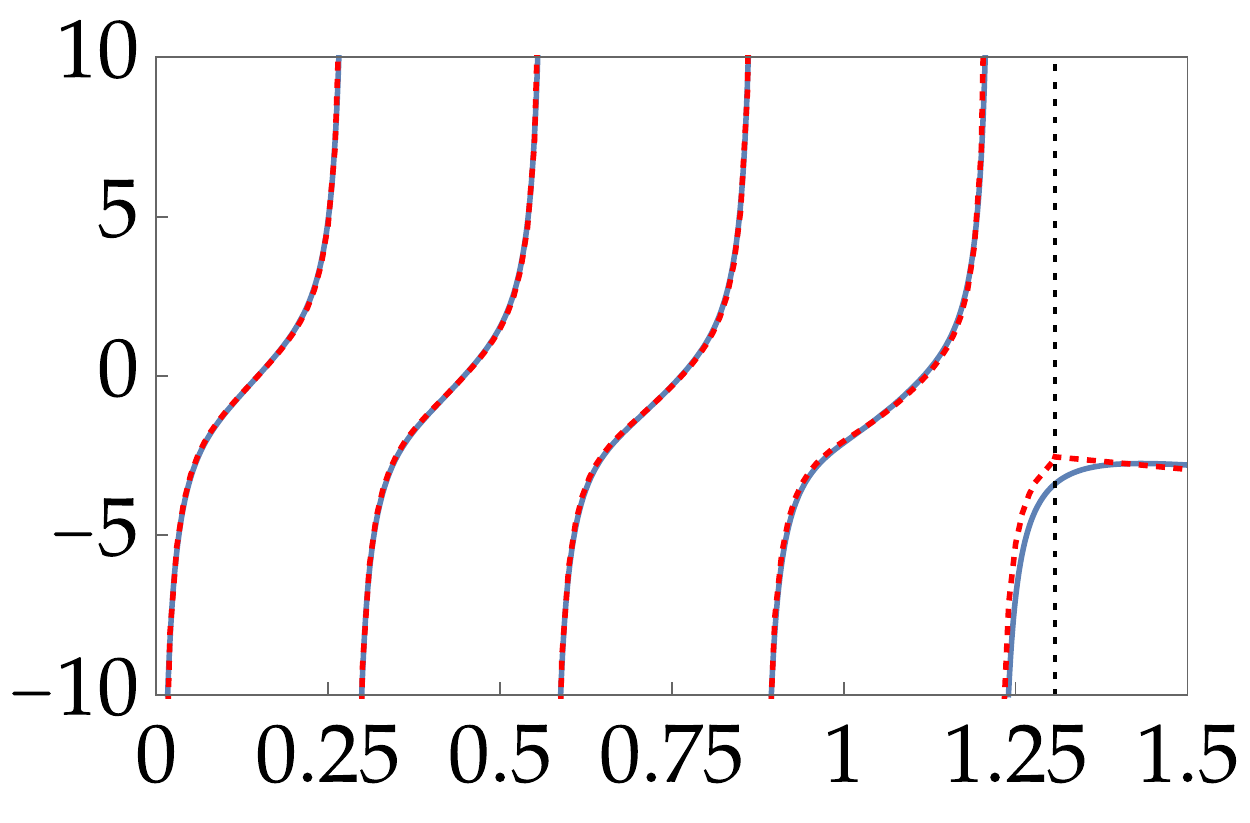}\\
\includegraphics[height=1.0 in]{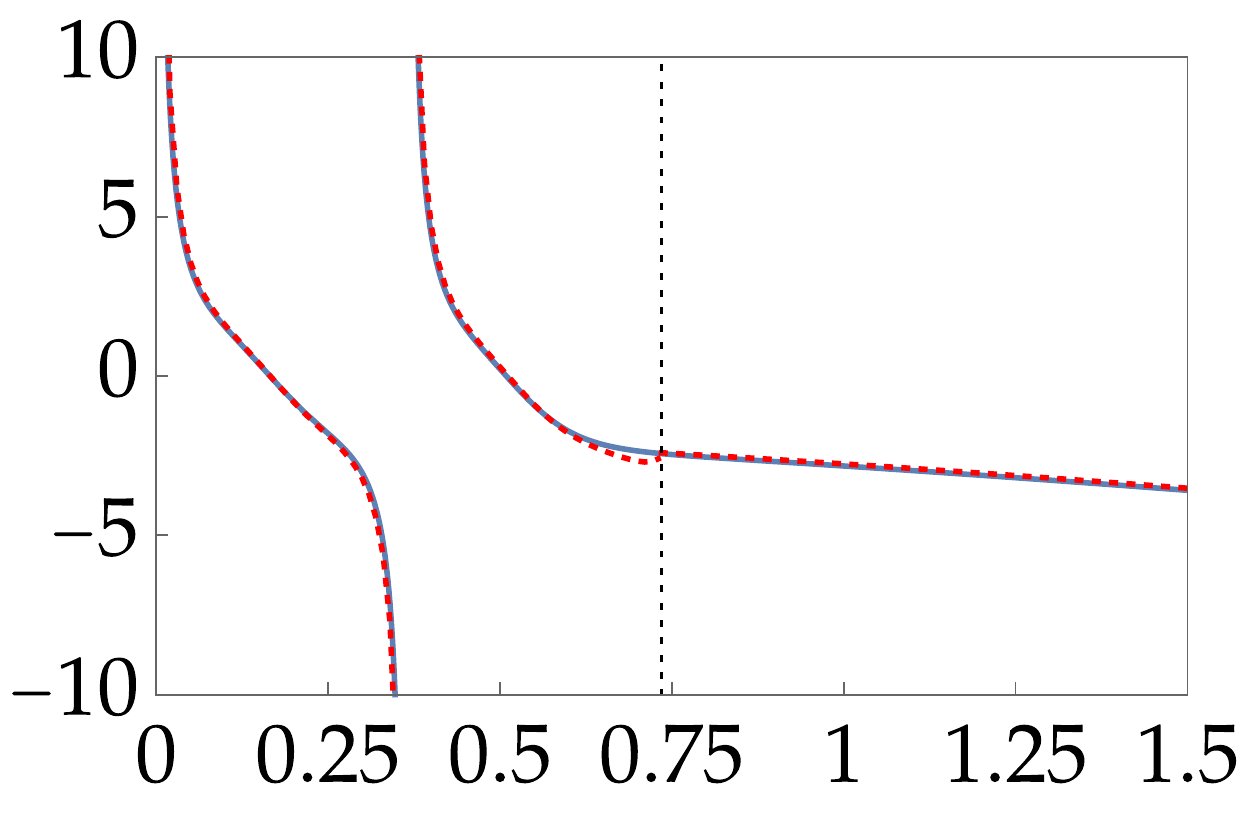}
\end{center}
\caption{$U=T^{-\frac{1}{2}}u^{[3]}_\mathrm{gH}(x;8,3)$; \\$\rho=\frac{3}{8}$; $\kappa=\tfrac{1}{3}$.}
\end{subfigure}
\end{center}
\caption{Quantitative comparison of scaled gH rational solutions $U$ (blue curves) with the leading terms $\dot{U}$ (dashed red curves) in Theorems~\ref{thm:HermiteExterior} and \ref{thm:Hermite-elliptic} for $x=T^\frac{1}{2}y_0$ (taking $\zeta=0$ fixed in the latter case).  Axes for top row of plots:  $U$ and $\dot{U}$ versus $y_0$.  Axes for bottom row of plots:  $-\ii U$ and $-\ii\dot{U}$ versus $-\ii y_0$.  Dotted vertical lines indicate the intersection points of $\partial\HermiteExterior(\kappa)$ with the coordinate axes, near which neither Theorem~\ref{thm:HermiteExterior} nor \ref{thm:Hermite-elliptic} provides a uniformly accurate approximation.}
\label{fig:gH-axes}
\end{figure}
\begin{figure}[h]
\begin{center}
\begin{subfigure}{1.5 in}
\begin{center}
\includegraphics[height=1.0 in]{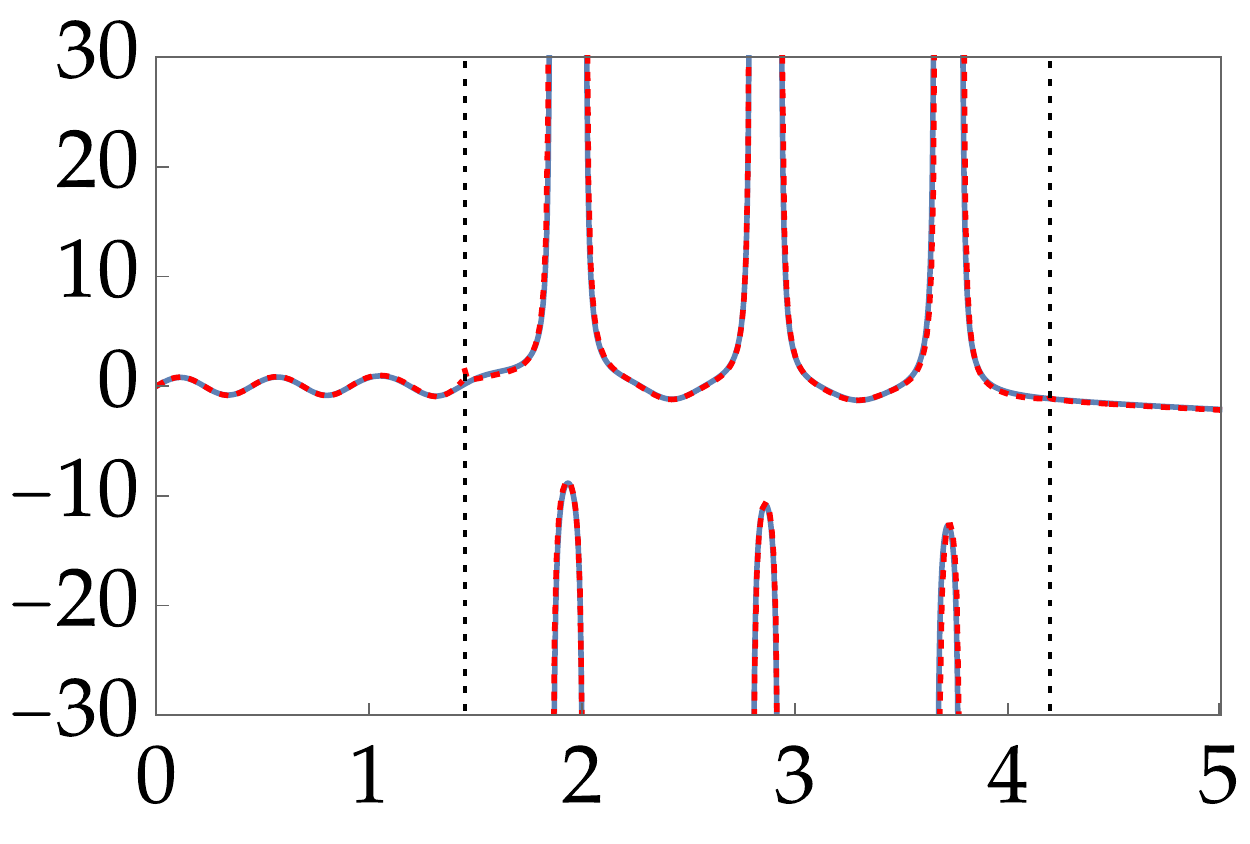}\\
\includegraphics[height=1.0 in]{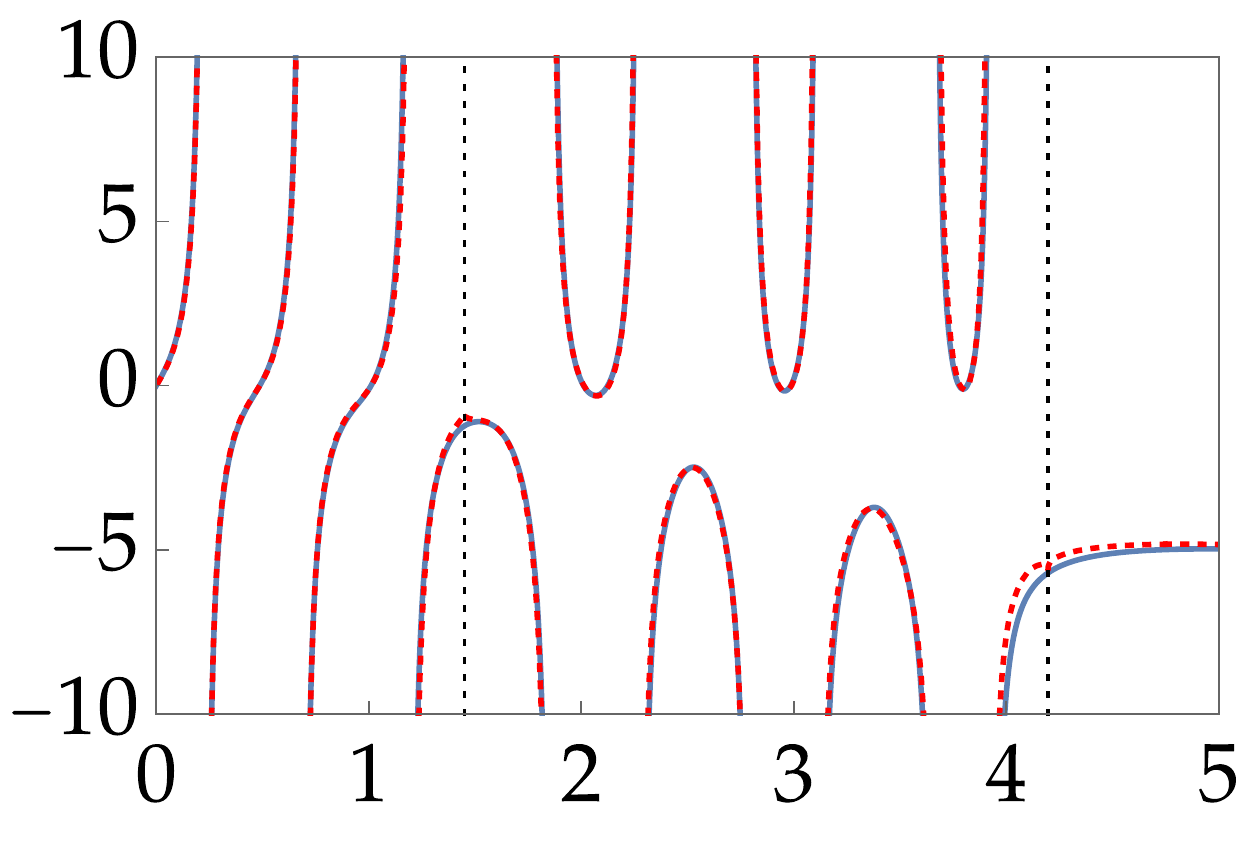}
\end{center}
\caption{$U=T^{-\frac{1}{2}}u^{[1]}_\mathrm{gO}(x;6,6)$; \\$\rho=1$; $\kappa=-\tfrac{57}{17}$.}
\end{subfigure}%
\begin{subfigure}{1.5 in}
\begin{center}
\includegraphics[height=1.0 in]{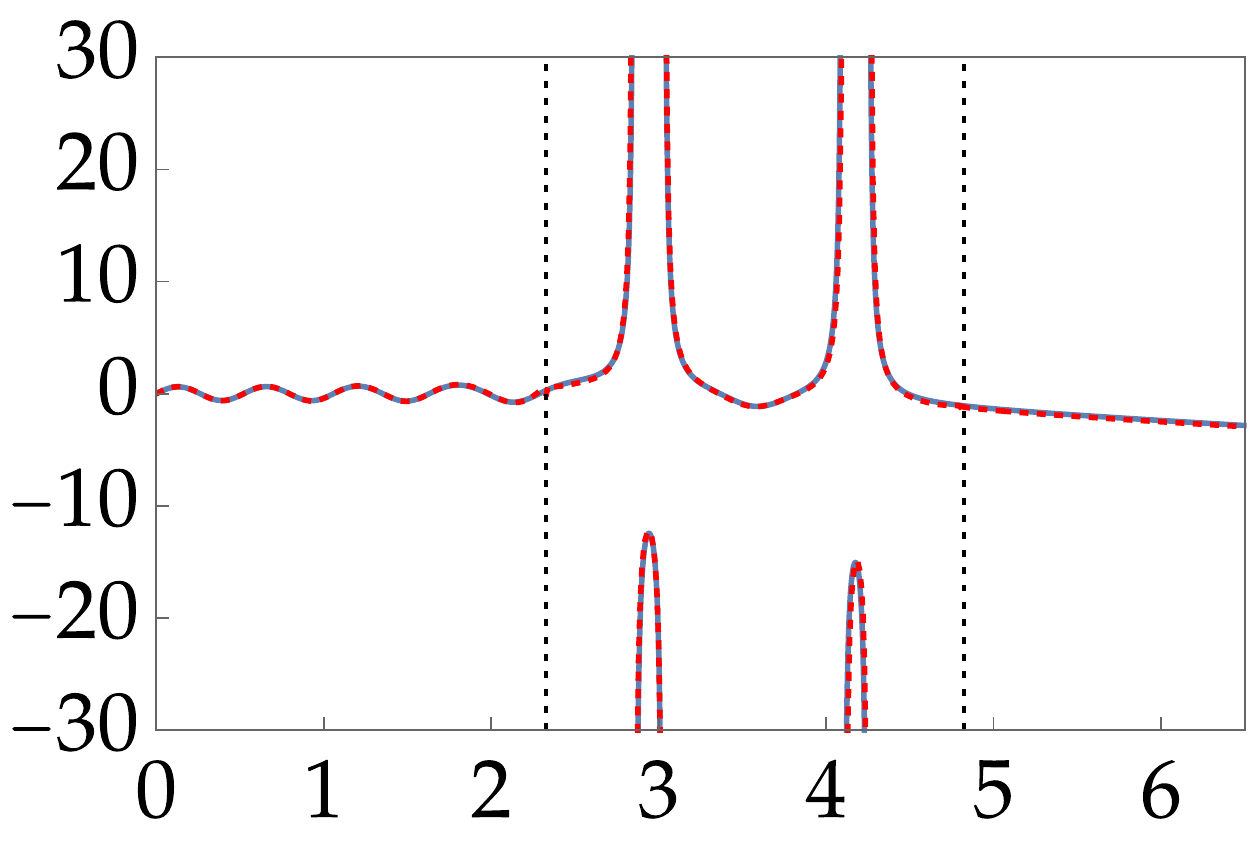}\\
\includegraphics[height=1.0 in]{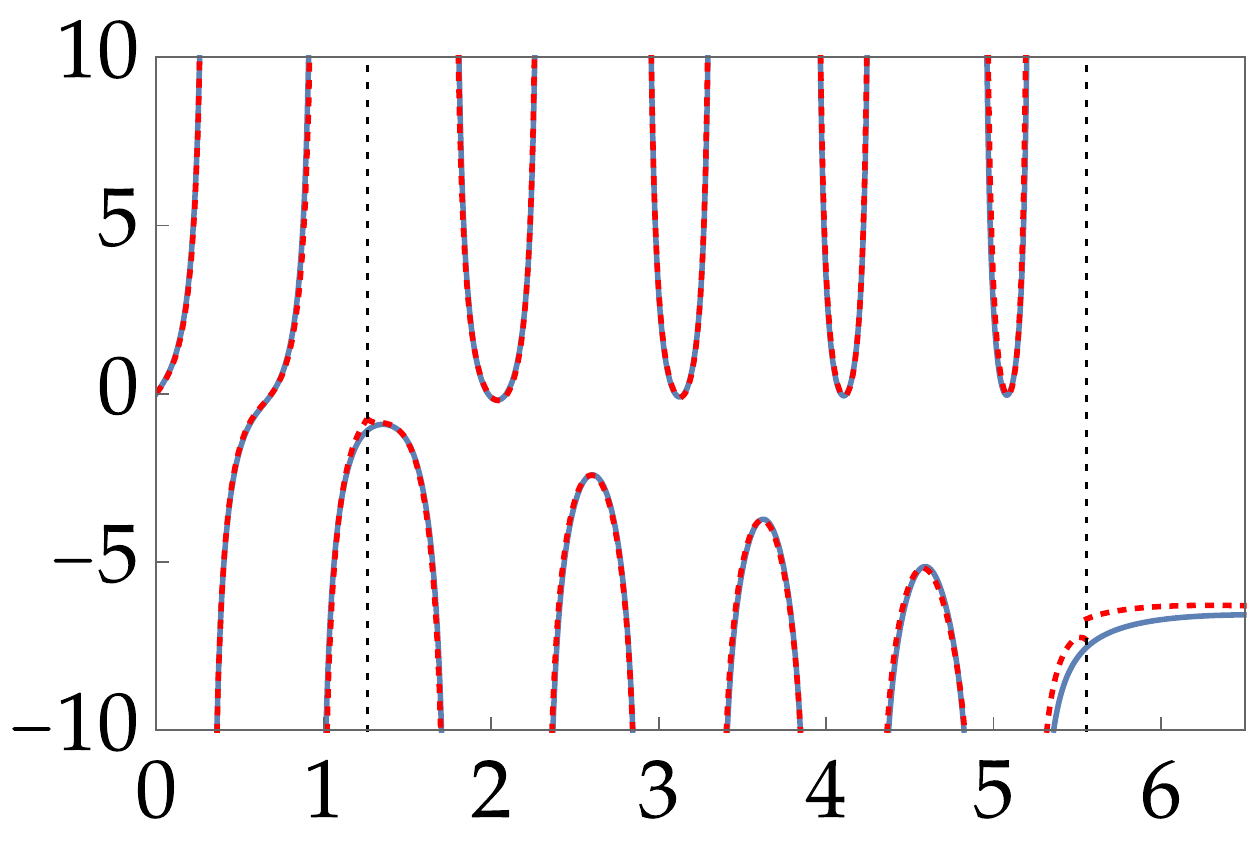}
\end{center}
\caption{$U=T^{-\frac{1}{2}}u^{[1]}_\mathrm{gO}(x;8,4)$; \\$\rho=\tfrac{1}{2}$; $\kappa=-\tfrac{63}{11}$.}
\end{subfigure}%
\begin{subfigure}{1.5 in}
\begin{center}
\includegraphics[height=1.0 in]{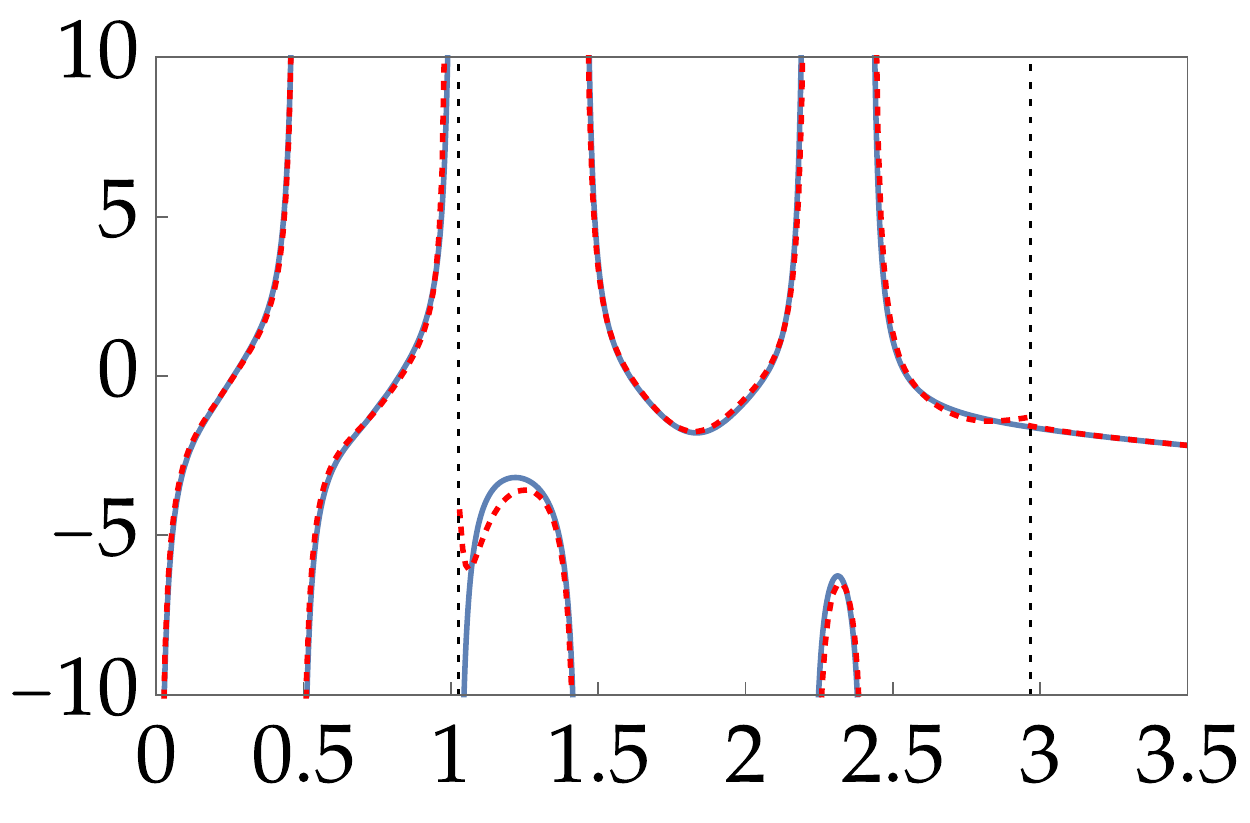}\\
\includegraphics[height=1.0 in]{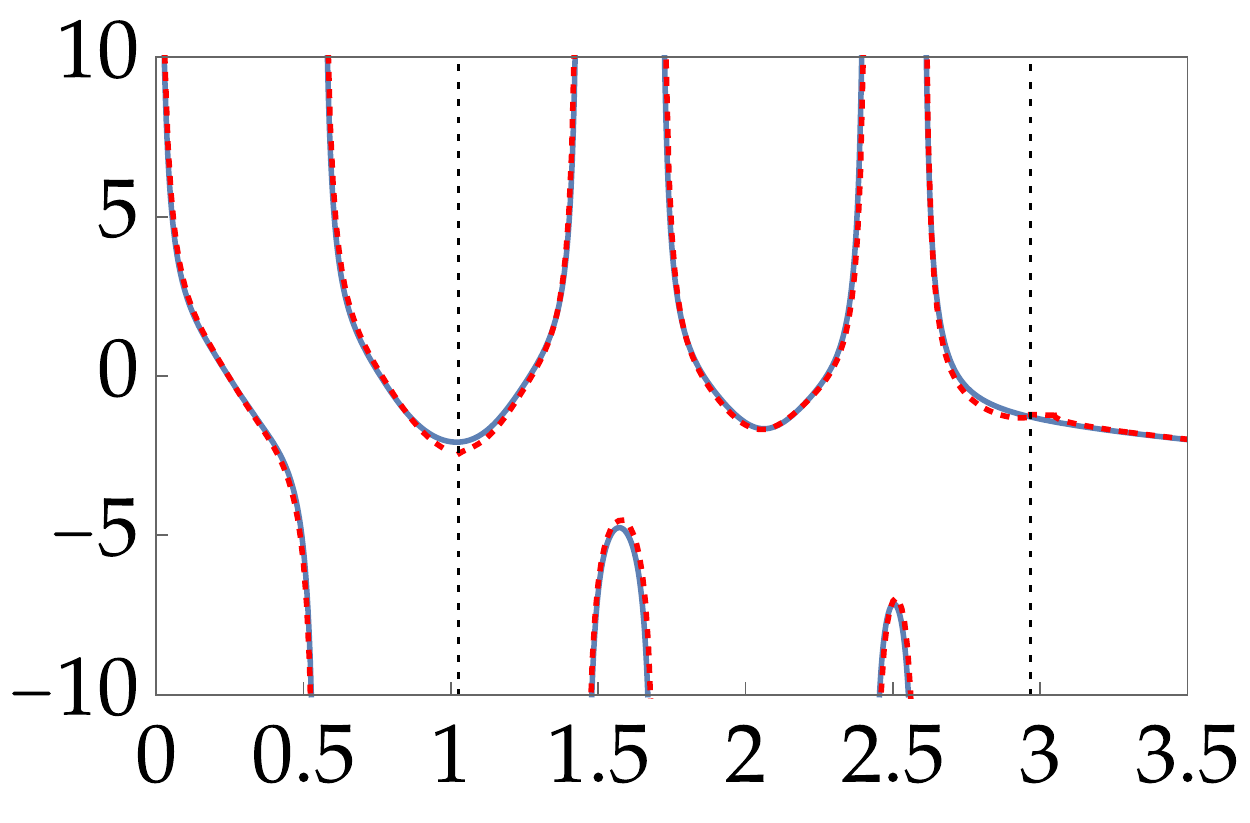}
\end{center}
\caption{$U=T^{-\frac{1}{2}}u^{[3]}_\mathrm{gO}(x;4,3)$; \\$\rho=\tfrac{3}{4}$; $\kappa=0$.}
\end{subfigure}%
\begin{subfigure}{1.5 in}
\begin{center}
\includegraphics[height=1.0 in]{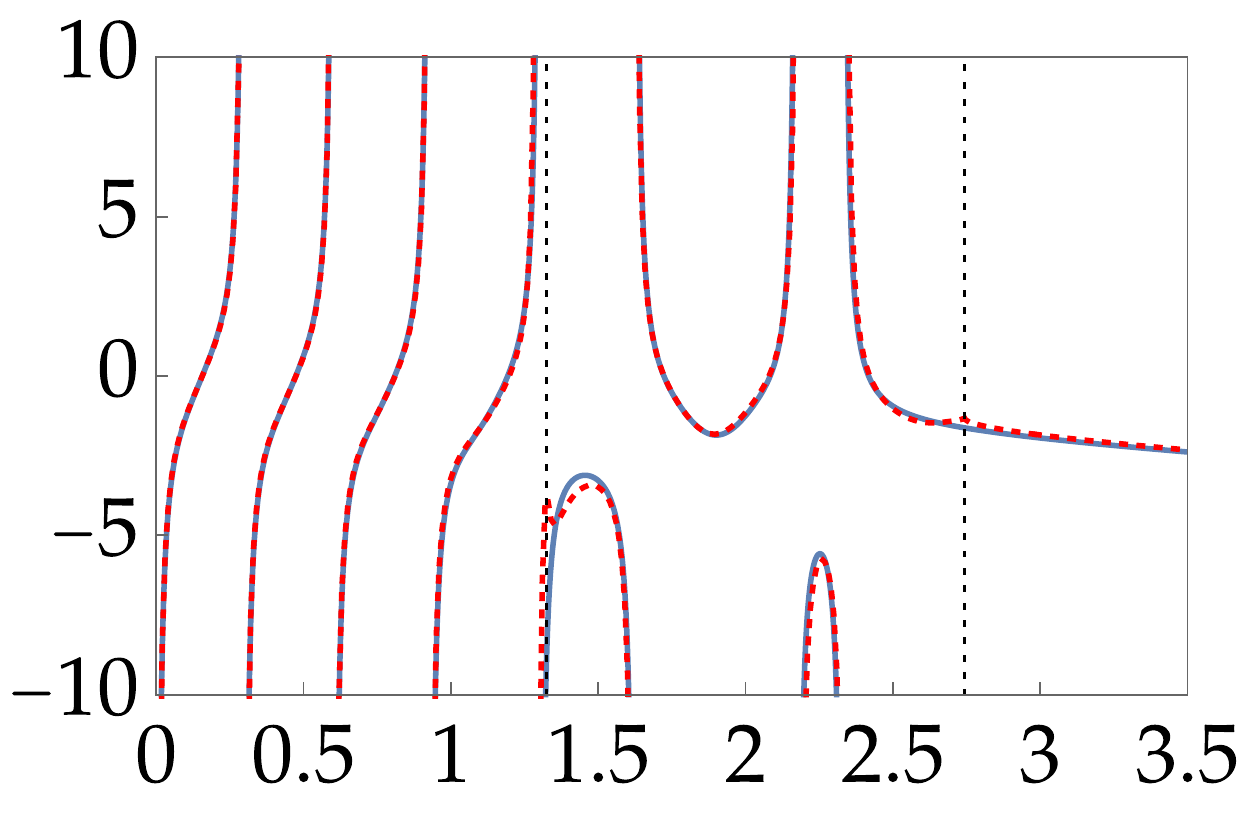}\\
\includegraphics[height=1.0 in]{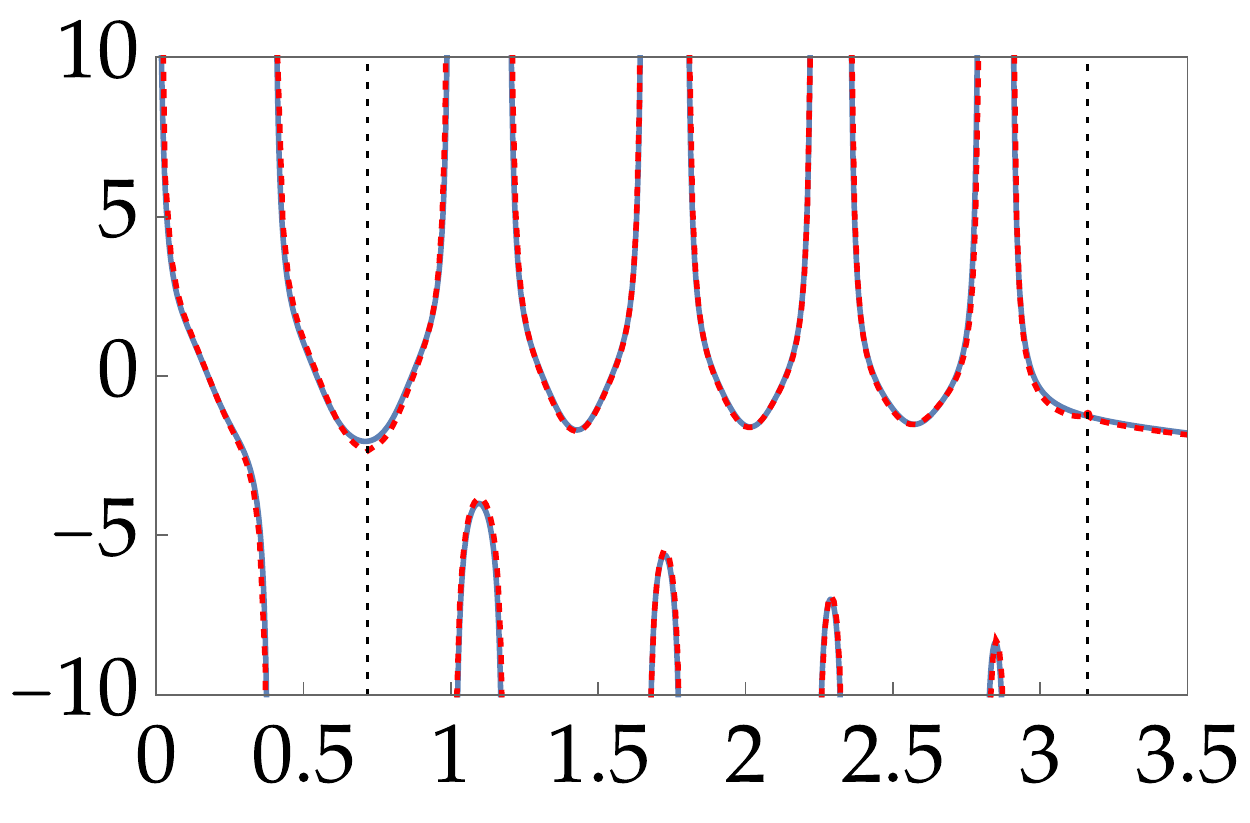}
\end{center}
\caption{$U=T^{-\frac{1}{2}}u^{[3]}_\mathrm{gO}(x;8,3)$; \\$\rho=\tfrac{3}{8}$; $\kappa=\tfrac{6}{17}$.}
\end{subfigure}
\end{center}
\caption{Quantitative comparison of scaled gO rational solutions $U$ (blue curves) for positive indices $(m,n)$ with the leading terms $\dot{U}$ (dashed red curves) in Theorems~\ref{thm:OkamotoExterior} and \ref{thm:Okamoto-elliptic} for $x=T^\frac{1}{2}y_0$ (taking $\zeta=0$ fixed in the latter case).  Axes for top row of plots:  $U$ and $\dot{U}$ versus $y_0$.  Axes for bottom row of plots:  $-\ii U$ and $-\ii\dot{U}$ versus $-\ii y_0$.  Dotted vertical lines indicate the intersection points of $\partial\HermiteExterior(\kappa)$ and $\partial\OkamotoExterior(\kappa)$ with the coordinate axes, near which neither Theorem~\ref{thm:OkamotoExterior} nor \ref{thm:Okamoto-elliptic} provides a uniformly accurate approximation.}
\label{fig:gOpp-axes}
\end{figure}
\begin{figure}[h]
\begin{center}
\begin{subfigure}{1.5 in}
\begin{center}
\includegraphics[height=1.0 in]{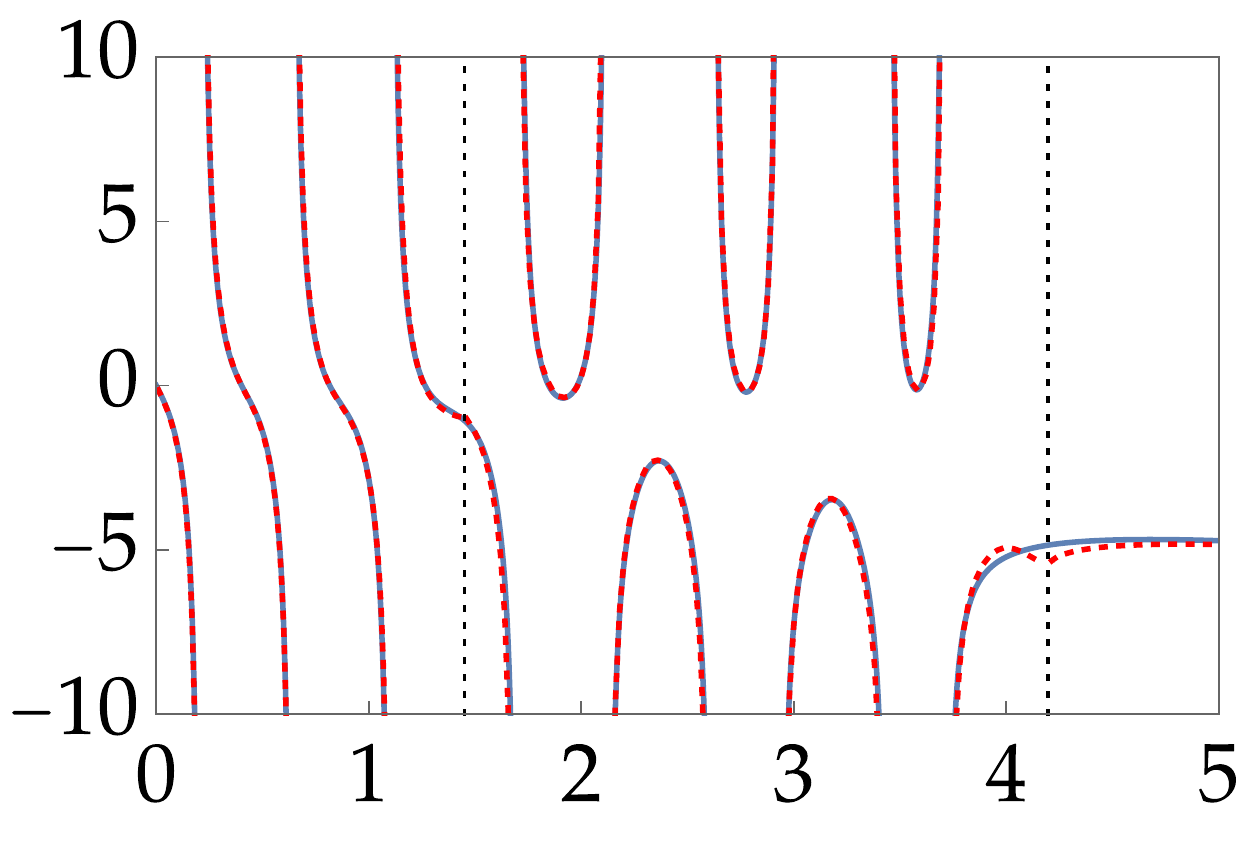}\\
\includegraphics[height=1.0 in]{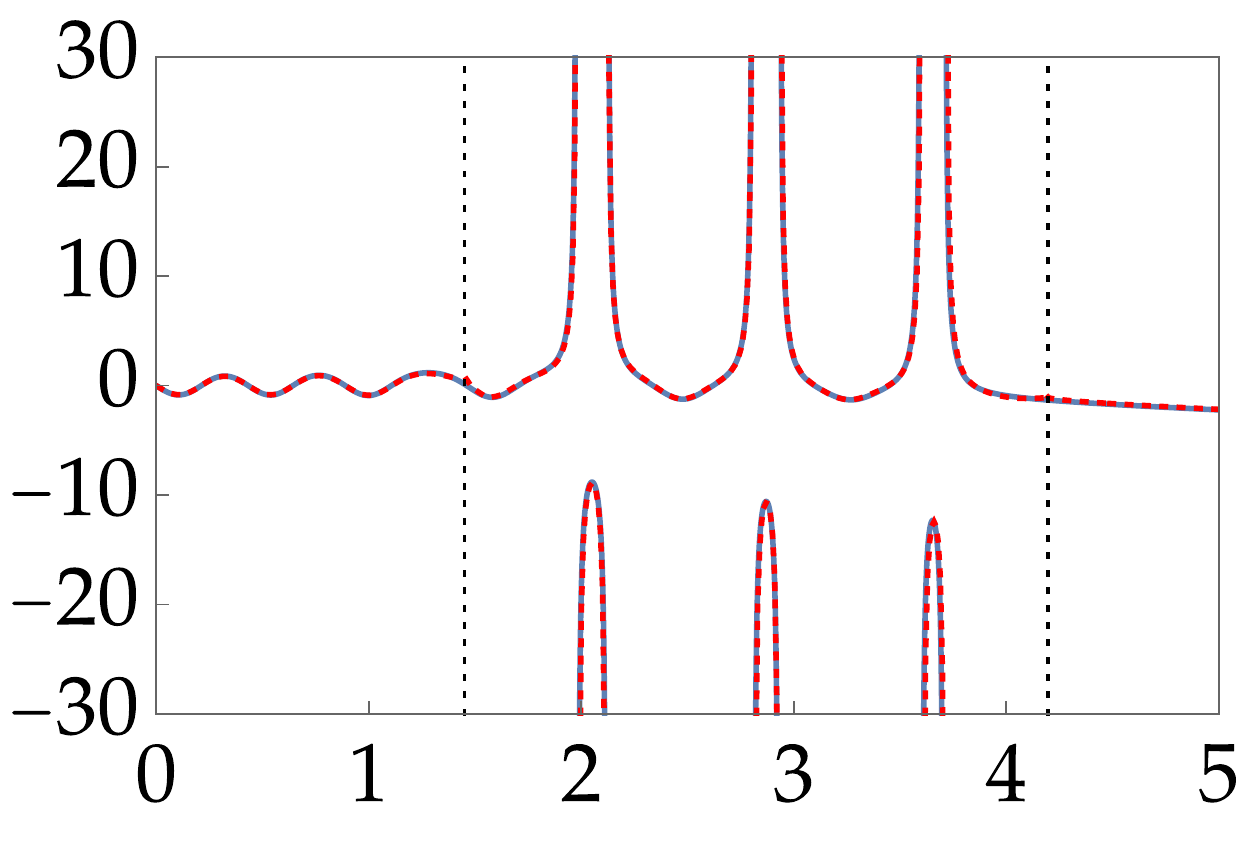}
\end{center}
\caption{$U=T^{-\frac{1}{2}}u^{[1]}_\mathrm{gO}(x;-6,-6)$; \\$\rho=1$; $\kappa=\tfrac{51}{19}$.}
\end{subfigure}%
\begin{subfigure}{1.5 in}
\begin{center}
\includegraphics[height=1.0 in]{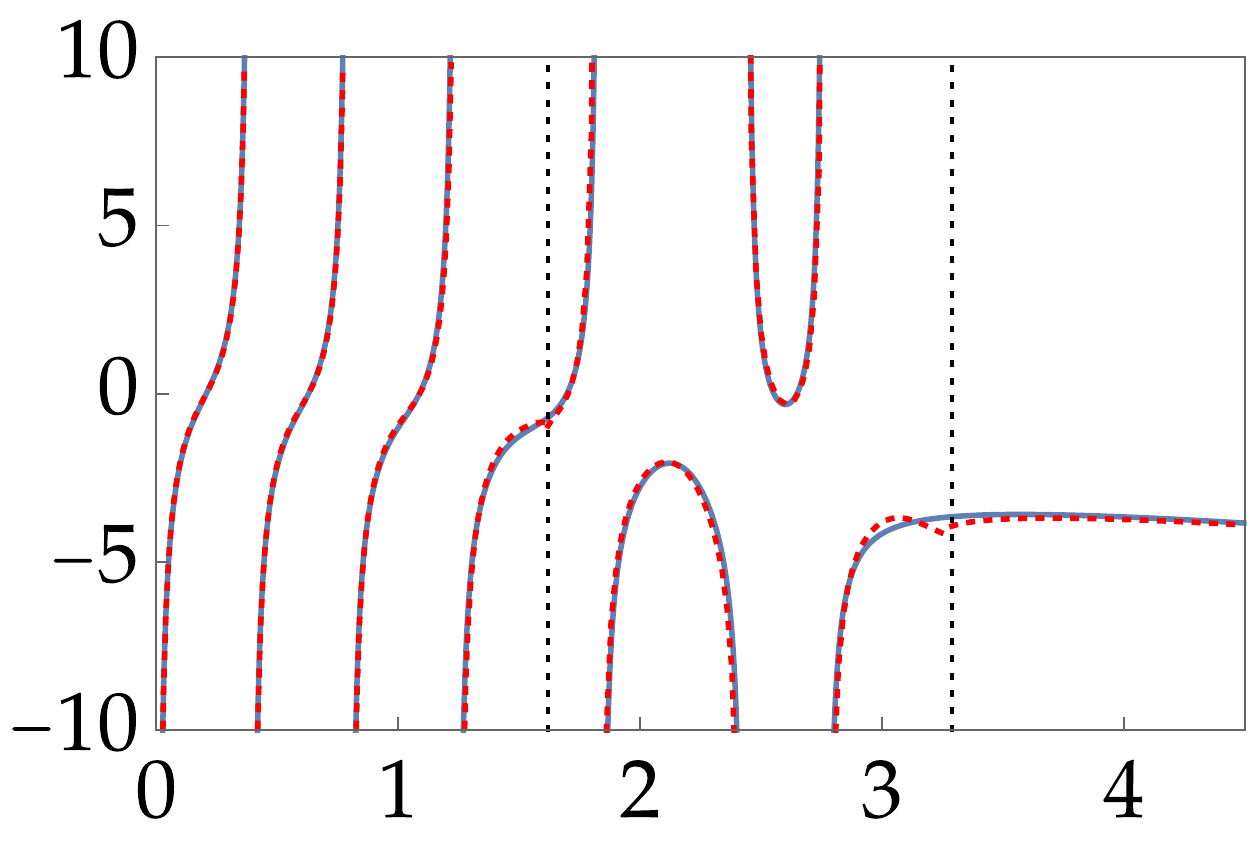}\\
\includegraphics[height=1.0 in]{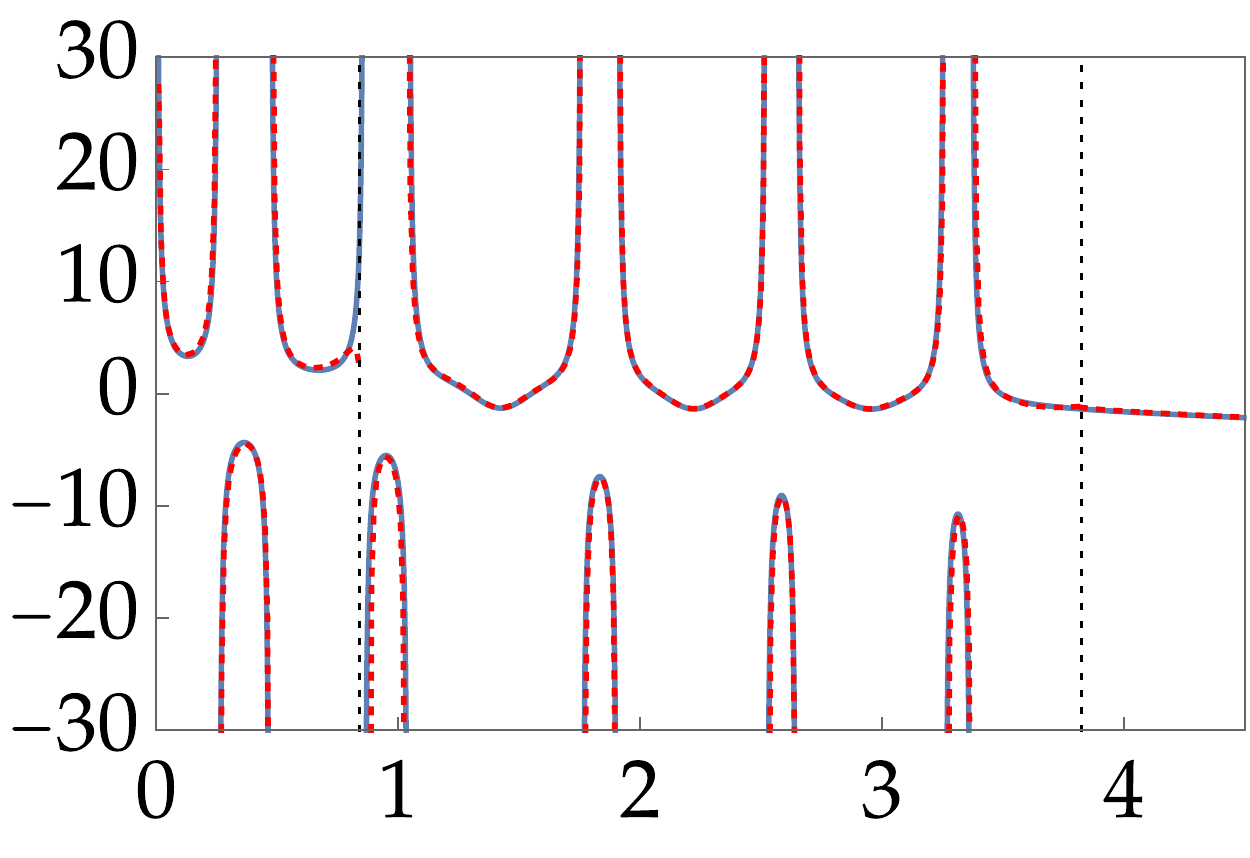}
\end{center}
\caption{$U=T^{-\frac{1}{2}}u^{[1]}_\mathrm{gO}(x;-3,-7)$; \\$\rho=\tfrac{7}{3}$; $\kappa=\tfrac{18}{11}$.}
\end{subfigure}%
\begin{subfigure}{1.5 in}
\begin{center}
\includegraphics[height=1.0 in]{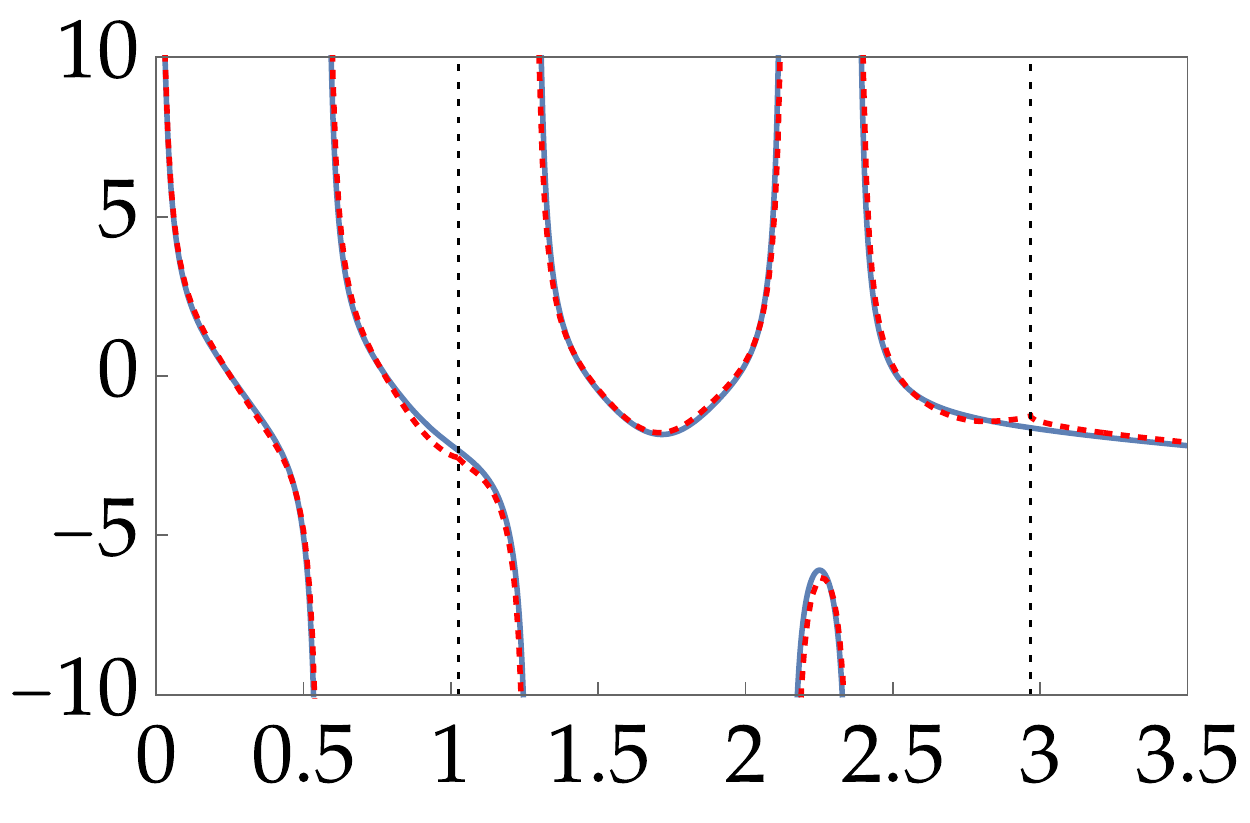}\\
\includegraphics[height=1.0 in]{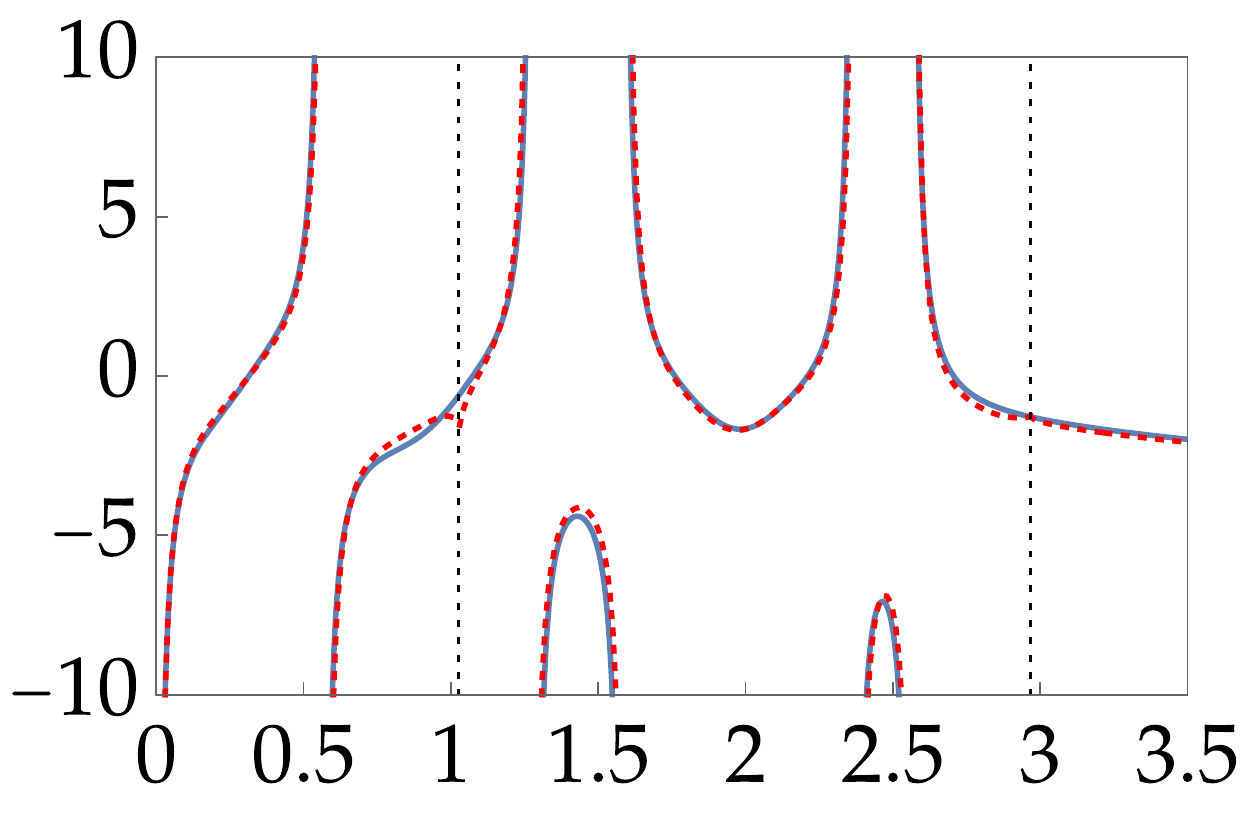}
\end{center}
\caption{$U=T^{-\frac{1}{2}}u^{[3]}_\mathrm{gO}(x;-3,-4)$; \\$\rho=\tfrac{4}{3}$; $\kappa=0$.}
\end{subfigure}
\begin{subfigure}{1.5 in}
\begin{center}
\includegraphics[height=1.0 in]{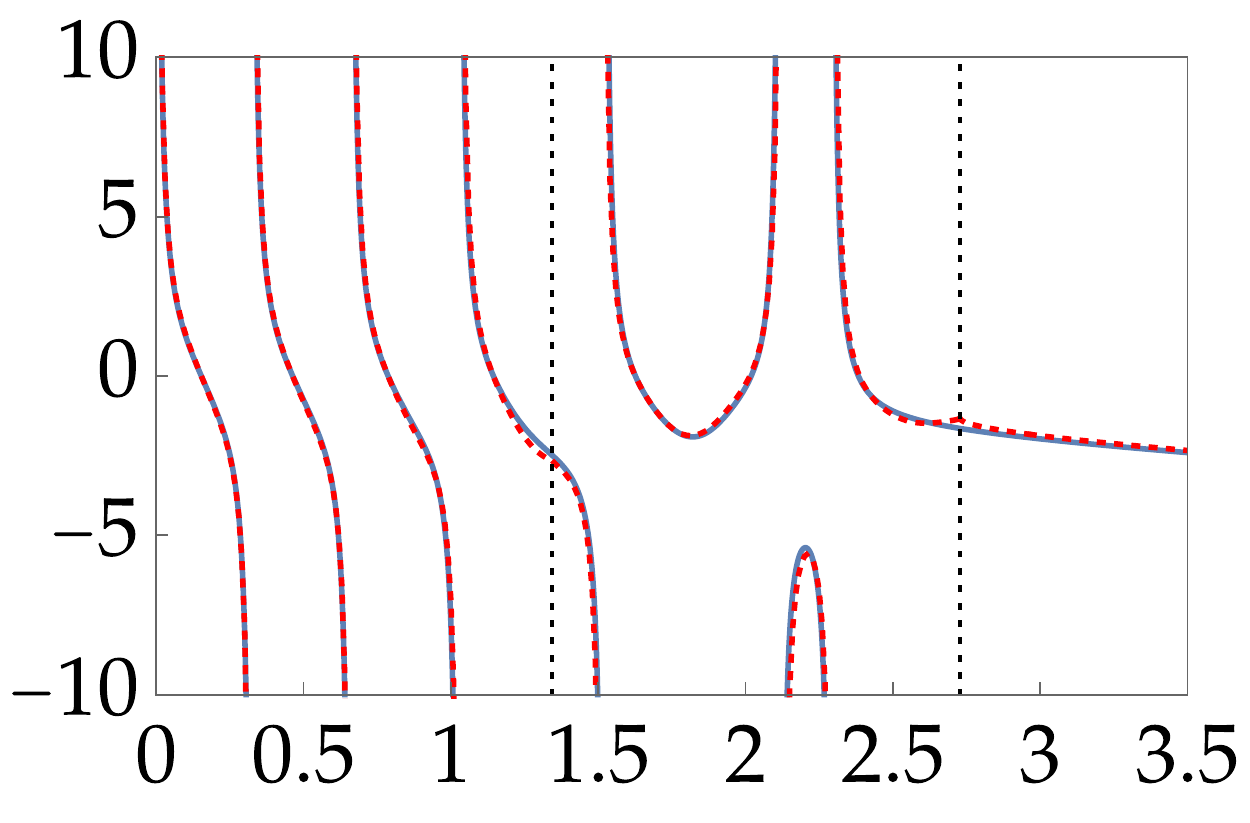}\\
\includegraphics[height=1.0 in]{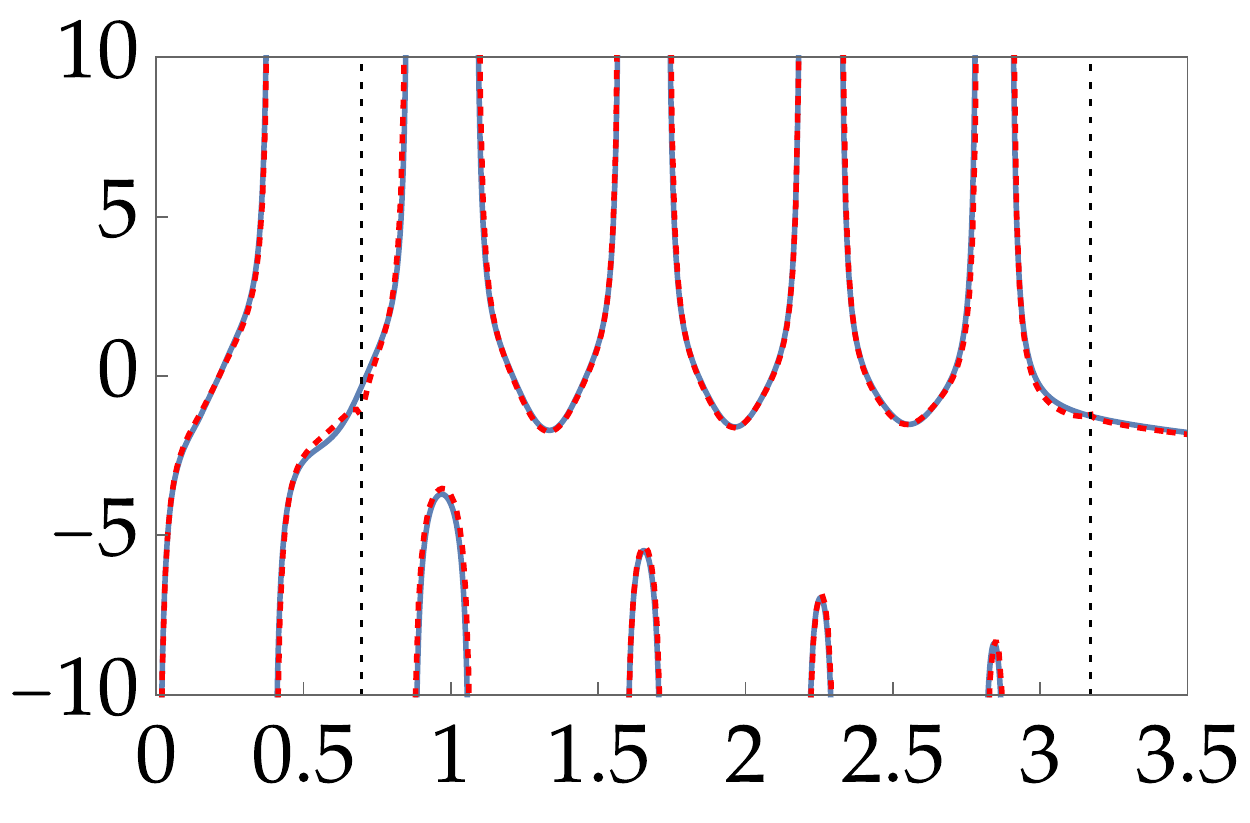}
\end{center}
\caption{$U=T^{-\frac{1}{2}}u^{[3]}_\mathrm{gO}(x;-3,-8)$; \\$\rho=\tfrac{8}{3}$; $\kappa=\tfrac{3}{8}$.}
\end{subfigure}%
\end{center}
\caption{As in Figure~\ref{fig:gOpp-axes} but for negative indices $(m,n)$.}
\label{fig:gOnn-axes}
\end{figure}
On the other hand, the approach (ii) allows one to accurately compare $U=T^{-\frac{1}{2}}u_\mathrm{F}^{[j]}(T^\frac{1}{2}y_0+T^{-\frac{1}{2}}\zeta;m,n)$ with an exact elliptic function $\dot{U}=f(\zeta-\zeta_0)$ of $\zeta$, by fixing a point $y_0$.  Theorems~\ref{thm:Hermite-elliptic}--\ref{thm:Okamoto-elliptic} predict the accuracy of such an approximation provided that $\zeta$ remains bounded.
% and avoids the Malgrange divisor.  
 Figure~\ref{fig:gO-zeta} illustrates the nature of convergence to such an exact elliptic approximation.  For given $(m,n)$, the approach (ii) approximation fails as $\zeta$ increases, just as the tangent space fails to approximate a curved base manifold except near the given base point.
\begin{figure}[h]
\begin{center}
\begin{subfigure}{2 in}
\begin{center}
\includegraphics[width=2 in]{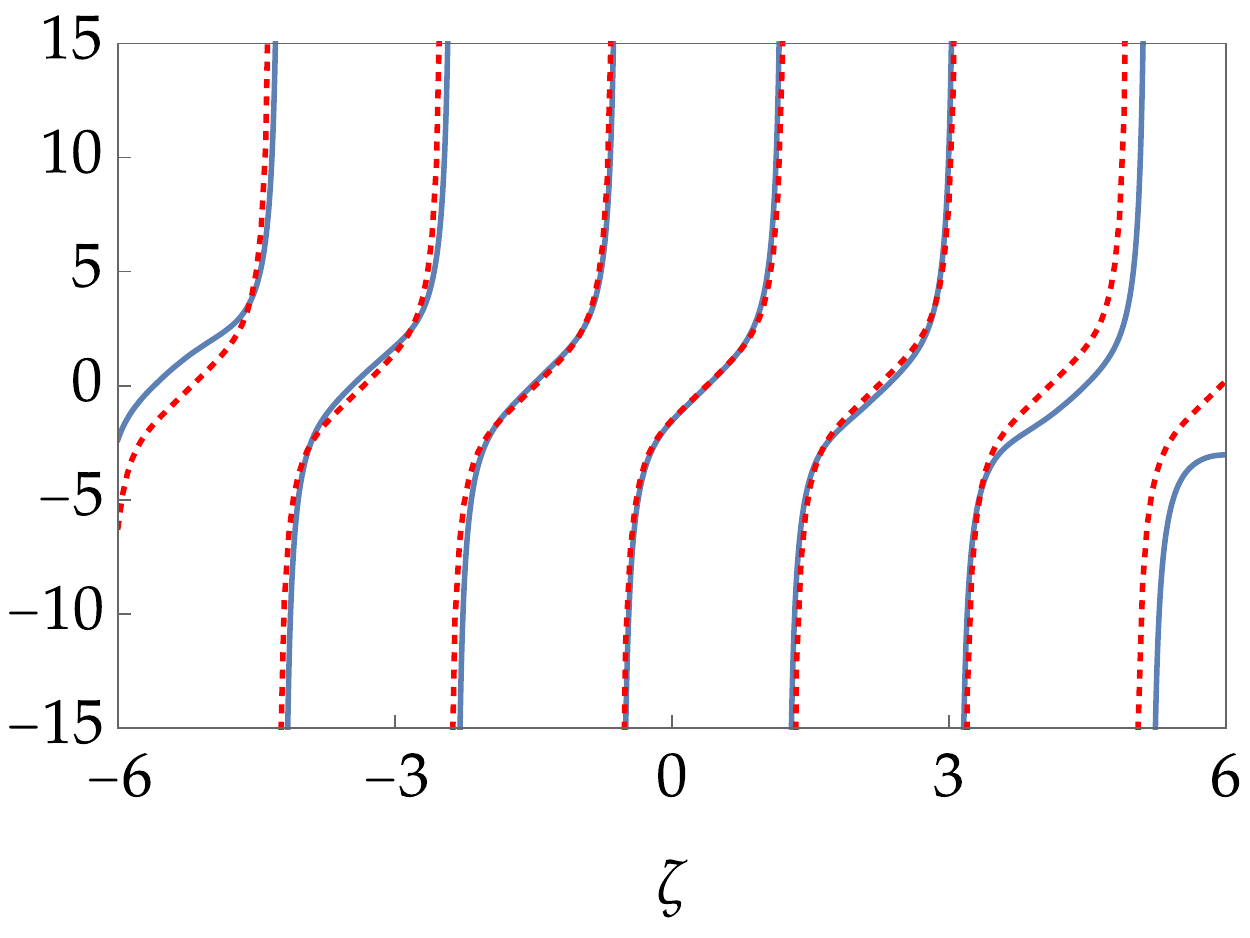}
\end{center}
\caption{$U=T^{-\frac{1}{2}}u_\mathrm{gO}^{[3]}(T^\frac{1}{2}(y_0+\frac{\zeta}{T});6,5).$}
\end{subfigure}%
\begin{subfigure}{2 in}
\begin{center}
\includegraphics[width=2 in]{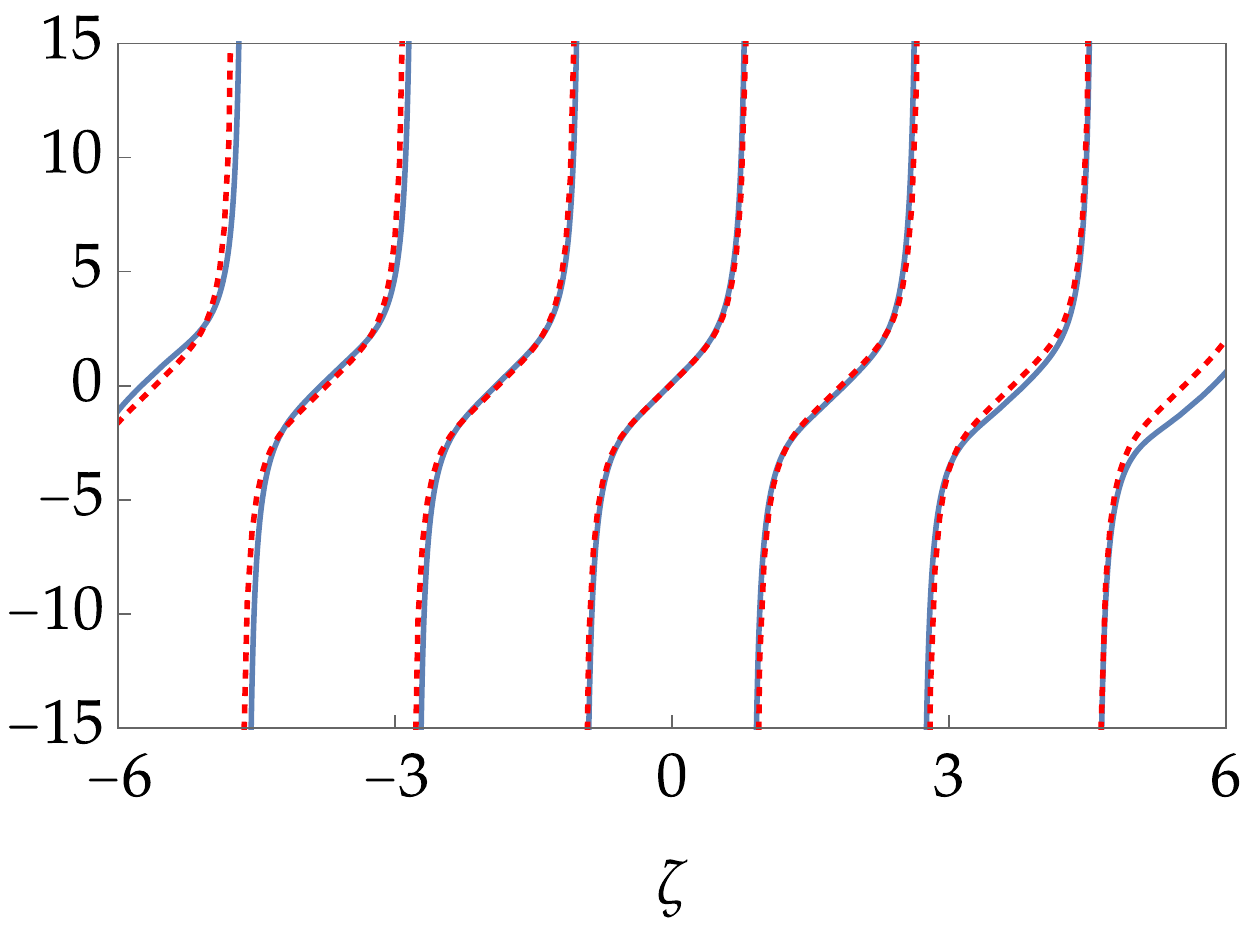}
\end{center}
\caption{$U=T^{-\frac{1}{2}}u_\mathrm{gO}^{[3]}(T^\frac{1}{2}(y_0+\frac{\zeta}{T});10,9).$}
\end{subfigure}%
\begin{subfigure}{2 in}
\begin{center}
\includegraphics[width=2 in]{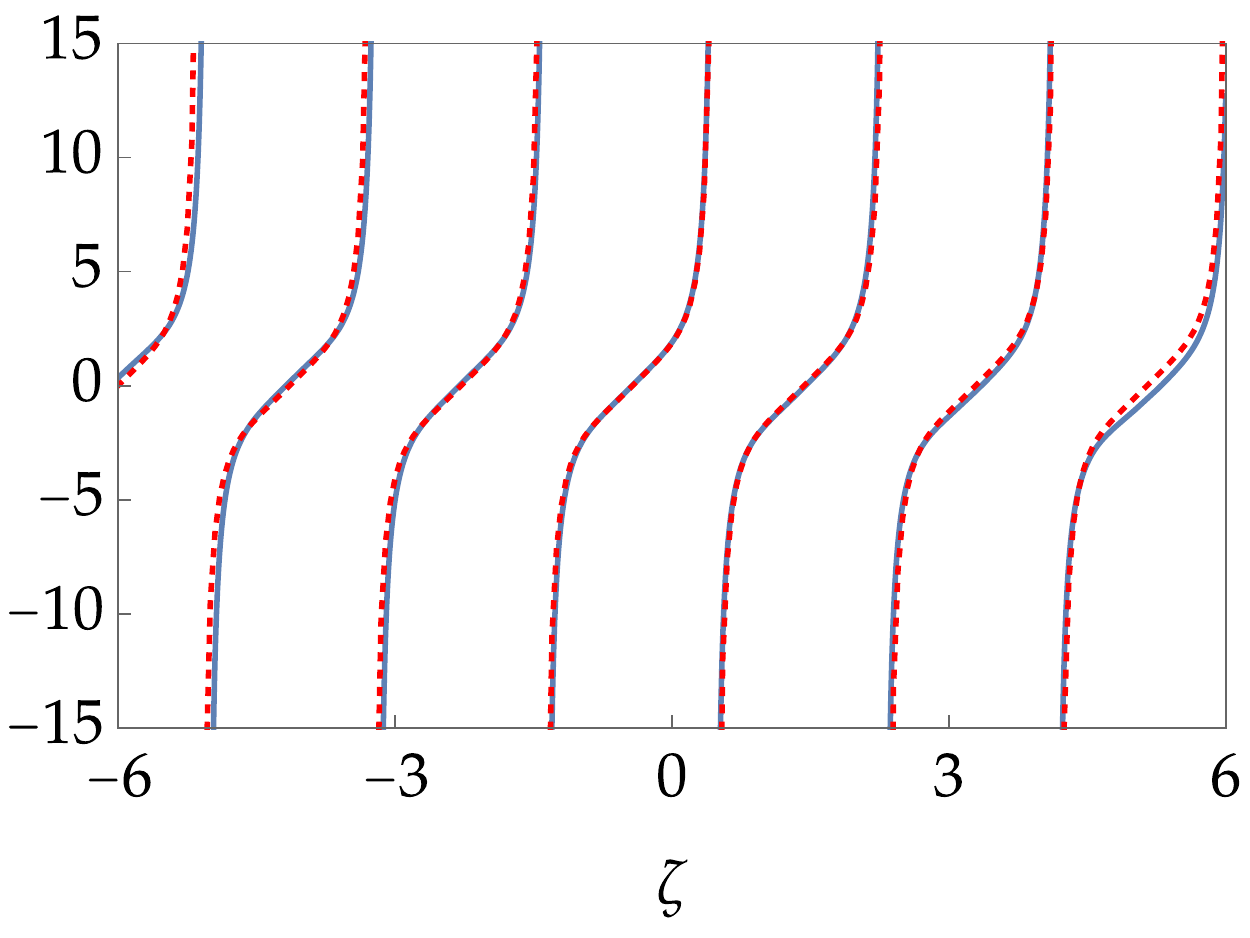}
\end{center}
\caption{$U=T^{-\frac{1}{2}}u_\mathrm{gO}^{[3]}(T^\frac{1}{2}(y_0+\frac{\zeta}{T});14,13).$}
\end{subfigure}%
\end{center}
\caption{Exact elliptic function approximations $\dot{U}=f(\zeta-\zeta_0)$ (dashed red curves) and the rational solutions $U$ they approximate (blue curves) near $y_0=\tfrac{1}{10}\in\rectangle(\kappa)\cap\mathbb{R}$ as functions of $\zeta\in\mathbb{R}$.  For all three plots, $\kappa=0$.}
\label{fig:gO-zeta}
\end{figure}

A good application of switching back and forth between the base manifold and its tangent space at a point is the proof of the following corollary of Theorems~\ref{thm:Hermite-elliptic} and \ref{thm:Okamoto-elliptic} (and of their proof, which specifies the various phases in \eqref{eq:intro-elliptic-theta}).  Note that according to \eqref{eq:intro-elliptic-theta}, 
the zeros of $f(\zeta-\zeta_0)$ are given by, for $k=1$ or $k=2$, the quantization conditions
\eq
\frac{\mathrm{Im}(Z_\mathfrak{b}^*\cdot(\zeta+\frac{Z_\mathfrak{a}}{2\pi}(\xi(y_0;m,n)+\mathfrak{z}_k(y_0))))}{\mathrm{Im}(Z_\mathfrak{a}^*Z_\mathfrak{b})}\in\mathbb{Z}\quad\text{and}\quad
-\frac{\mathrm{Im}(Z_\mathfrak{a}^*\cdot(\zeta+\frac{Z_\mathfrak{a}}{2\pi}(\xi(y_0;m,n)+\mathfrak{z}_k(y_0))))}{\mathrm{Im}(Z_\mathfrak{a}^*Z_\mathfrak{b})}\in\mathbb{Z}.
\label{eq:intro-zeros-quantize}
\endeq
Likewise, the poles of $f(\zeta-\zeta_0)$ are given by, for $k=1$ or $k=2$, the quantization conditions
\eq
\frac{\mathrm{Im}(Z_\mathfrak{b}^*\cdot(\zeta+\frac{Z_\mathfrak{a}}{2\pi}(\xi(y_0;m,n)+\mathfrak{p}_k(y_0))))}{\mathrm{Im}(Z_\mathfrak{a}^*Z_\mathfrak{b})}\in\mathbb{Z}\quad\text{and}\quad
-\frac{\mathrm{Im}(Z_\mathfrak{a}^*\cdot(\zeta+\frac{Z_\mathfrak{a}}{2\pi}(\xi(y_0;m,n)+\mathfrak{p}_k(y_0))))}{\mathrm{Im}(Z_\mathfrak{a}^*Z_\mathfrak{b})}\in\mathbb{Z}.
\label{eq:intro-poles-quantize}
\endeq
The phase $\xi(y_0;m,n)$ is given below in \eqref{eq:F1-U}, in which the dependence on the family (gH or gO), domain $\rectangle(\kappa)$, $\TR(\kappa)$, or $\TI(\kappa)$, and a sign $s$ ($s=-\mathrm{sgn}(\Theta_0)$ for type-$1$ solutions and $s=\mathrm{sgn}(\Theta_0)$ for type-$3$ solutions) enters via the data in Table~\ref{tab:outer-uniformize-phases}.  In particular, $\xi(y_0;m,n)$ contains terms proportional via $T\gg 1$ to real quantities $R_1$ and $R_2$ (see \eqref{eq:BoutrouxConstants}) that are essentially the imaginary parts of the integrals whose real parts vanish in \eqref{eq:intro-Boutroux}.  The phase shifts $\mathfrak{z}_k(y_0)$ and $\mathfrak{p}_k(y_0)$ depend on $y_0$ as well as the type of the rational solution.  They are written for type-$1$ solutions in \eqref{eq:dotUtwist-phases} and for type-$3$ solutions in \eqref{eq:dotU-phases}.
Define, for either family $\mathrm{F}=\mathrm{gH}$ or $\mathrm{F}=\mathrm{gO}$, and types $j=1,2,3$,
\eq
\begin{split}
\mathcal{Z}_\mathrm{F}^{[j]}(m,n)&:=\{y_0\in\mathbb{C}: u_\mathrm{F}^{[j]}(|\Theta_{0,\mathrm{F}}^{[j]}(m,n)|^\frac{1}{2}y_0;m,n)=0\}\\
\mathcal{P}_\mathrm{F}^{[j]}(m,n)&:=\{y_0\in\mathbb{C}: u_\mathrm{F}^{[j]}(|\Theta_{0,\mathrm{F}}^{[j]}(m,n)|^\frac{1}{2}y_0;m,n)=\infty\}
\end{split}
\label{eq:ExactRescaledZerosAndPoles}
\endeq
as the sets of rescaled zeros and poles of the indicated rational solution.  Likewise, set
\eq
\dot{\mathcal{Z}}_\mathrm{F}^{[j]}(m,n):=\{y_0\in\mathbb{C}: f(-\zeta_0)=0\}\quad\text{and}\quad
\dot{\mathcal{P}}_\mathrm{F}^{[j]}(m,n):=\{y_0\in\mathbb{C}: f(-\zeta_0)=\infty\}
\label{eq:ApproximateRescaledZerosAndPoles}
\endeq
where $f(\zeta-\zeta_0)$ is the approximation of the corresponding rational solution via Theorem~\ref{thm:Hermite-elliptic} or \ref{thm:Okamoto-elliptic}.  In other words, $\dot{\mathcal{Z}}_\mathrm{F}^{[j]}(m,n)$ (resp., $\dot{\mathcal{P}}_\mathrm{F}^{[j]}(m,n)$) is the set of all points $y_0$ satisfying both conditions in \eqref{eq:intro-zeros-quantize} (resp., in \eqref{eq:intro-poles-quantize}) with $\zeta=0$ fixed for either $k=1$ or $k=2$ and phases determined for the family, type, and region of interest.
%and the condition $\zeta-\zeta_0\in\mathcal{M}^{[j]}$ corresponds to exactly one of the cases: $k=1$ or $k=2$.
\begin{corollary}[Poles and zeros of Painlev\'e-IV rational solutions]
Fix a rational aspect ratio $\rho>0$ and a compact set $C$ within one of the domains $\rectangle(\kappa)$, $\TR(\kappa)$, or $\TI(\kappa)$ (the latter two for the gO family only).  
%Let $C^0(m,n)\subset C$ be the set of $y_0\in C$ for which both conditions in \eqref{eq:intro-zeros-quantize} hold with $\zeta=0$ fixed for either $k=1$ or $k=2$, i.e., $C^0(m,n)$ consists of all zeros $y_0\in C$ of $f(-\zeta_0)$.  Likewise, let $C^\infty(m,n)\subset C$ be the set of $y_0\in C$ for which both conditions in \eqref{eq:intro-poles-quantize} hold with $\zeta=0$ fixed for 
%either $k=1$ or $k=2$, i.e., $C^\infty(m,n)$ consists of all singularities $y_0\in C$ of $f(-\zeta_0)$.
Then, there is a constant $r>0$ (depending on $C$ and $\rho$) such that for 
$m,n$ sufficiently large 
with $n=\rho m$ the following statements hold with $T=|\Theta_{0,\mathrm{F}}^{[j]}(m,n)|$.
\begin{itemize} 
\item
For each point $\dot{y}_0\in \dot{\mathcal{Z}}^{[j]}_\mathrm{F}(m,n)\cap C$ there is a unique point $y_0\in \mathcal{Z}^{[j]}_\mathrm{F}(m,n)$ that satisfies $|y_0-\dot{y}_0|\le rT^{-2}$.  Likewise for each point $y_0\in\mathcal{Z}^{[j]}_\mathrm{F}(m,n)\cap C$ there is a unique point $\dot{y}_0\in\dot{\mathcal{Z}}^{[j]}_\mathrm{F}(m,n)$ that satisfies $|\dot{y}_0-y_0|\le r T^{-2}$.
\item
For each point $\dot{y}_0\in \dot{\mathcal{P}}^{[j]}_\mathrm{F}(m,n)\cap C$ there is a unique point $y_0\in \mathcal{P}^{[j]}_\mathrm{F}(m,n)$ that satisfies $|y_0-\dot{y}_0|\le rT^{-2}$.  Likewise for each point $y_0\in\mathcal{P}^{[j]}_\mathrm{F}(m,n)\cap C$ there is a unique point $\dot{y}_0\in\dot{\mathcal{P}}^{[j]}_\mathrm{F}(m,n)$ that satisfies $|\dot{y}_0-y_0|\le r T^{-2}$.
\end{itemize}
\label{cor:poles-and-zeros}
\end{corollary}
%\begin{remark}
%The proofs of the exact correspondence results use the analytic implicit function theorem, but this is not available if Theorems~\ref{thm:Hermite-elliptic} and \ref{thm:Okamoto-elliptic} do not allow the approximation of $u(x)$ in a neighborhood of a pole of $f(\zeta-\zeta_0)$.  So in the latter case, we use instead a winding number calculation, which admits the possibility that a cluster of excess poles and zeros converge toward a distinguished pole of $f(\zeta-\zeta_0)$.  If the condition in these Theorems that $\zeta-\zeta_0$ should be bounded away from $\mathcal{M}^{[j]}$ is removed, as discussed in Remark~\ref{rem:Malgrange}, then the analytic implicit function theorem applies equally to all poles and zeros of $f(\zeta-\zeta_0)$, and the clustering phenomenon is ruled out.  We have never observed any clustering to occur.  
%\label{rem:NoClusters}
%\end{remark}
The proof of Corollary~\ref{cor:poles-and-zeros} is given at the end of Section~\ref{sec:MalgrangeResidues} below.  The accuracy of approximation of poles and zeros of the rational Painlev\'e-IV solutions $u_\mathrm{F}^{[j]}(x;m,n)$ for both families $\mathrm{F}=\mathrm{gH}$ and $\mathrm{F}=\mathrm{gO}$ and types $j=1$ and $j=3$ are shown in Figures~\ref{fig:gH-poles-zeros}--\ref{fig:gOnn-poles-zeros} for $y_0$ in the closed first quadrant $0\le\arg(y_0)\le\tfrac{1}{2}\pi$.  Again, the accuracy is remarkable even for $(m,n)$ not very large.  We do not show analogous plots for type $j=2$ since these can be immediately obtained from \eqref{eq:symmetry-1-2} and Proposition~\ref{prop:FirstQuadrant}.
\begin{figure}[h]
\begin{center}
\begin{subfigure}{1.5 in}
\begin{center}
\includegraphics[height=1.5 in]{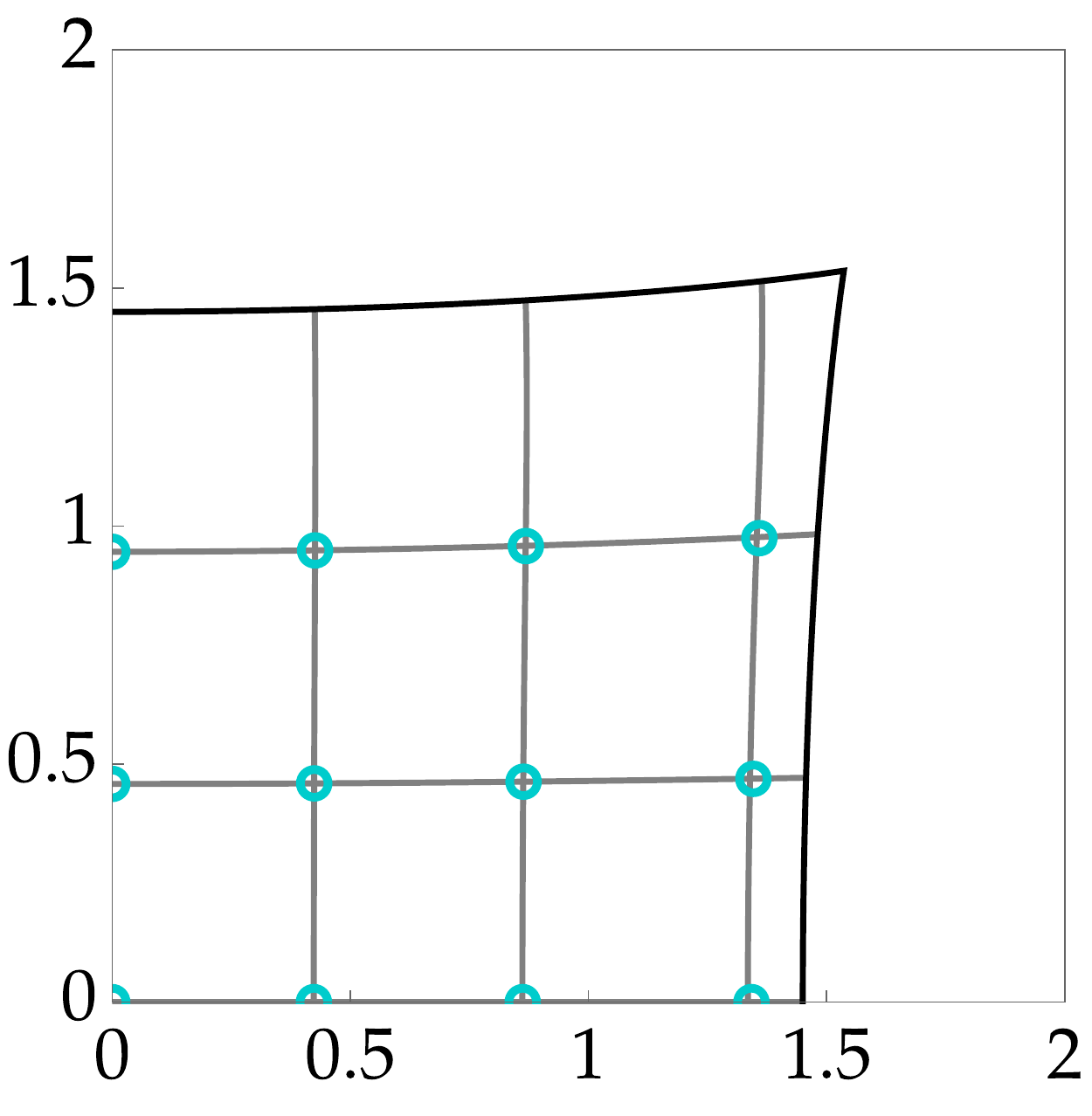}\\
\includegraphics[height=1.5 in]{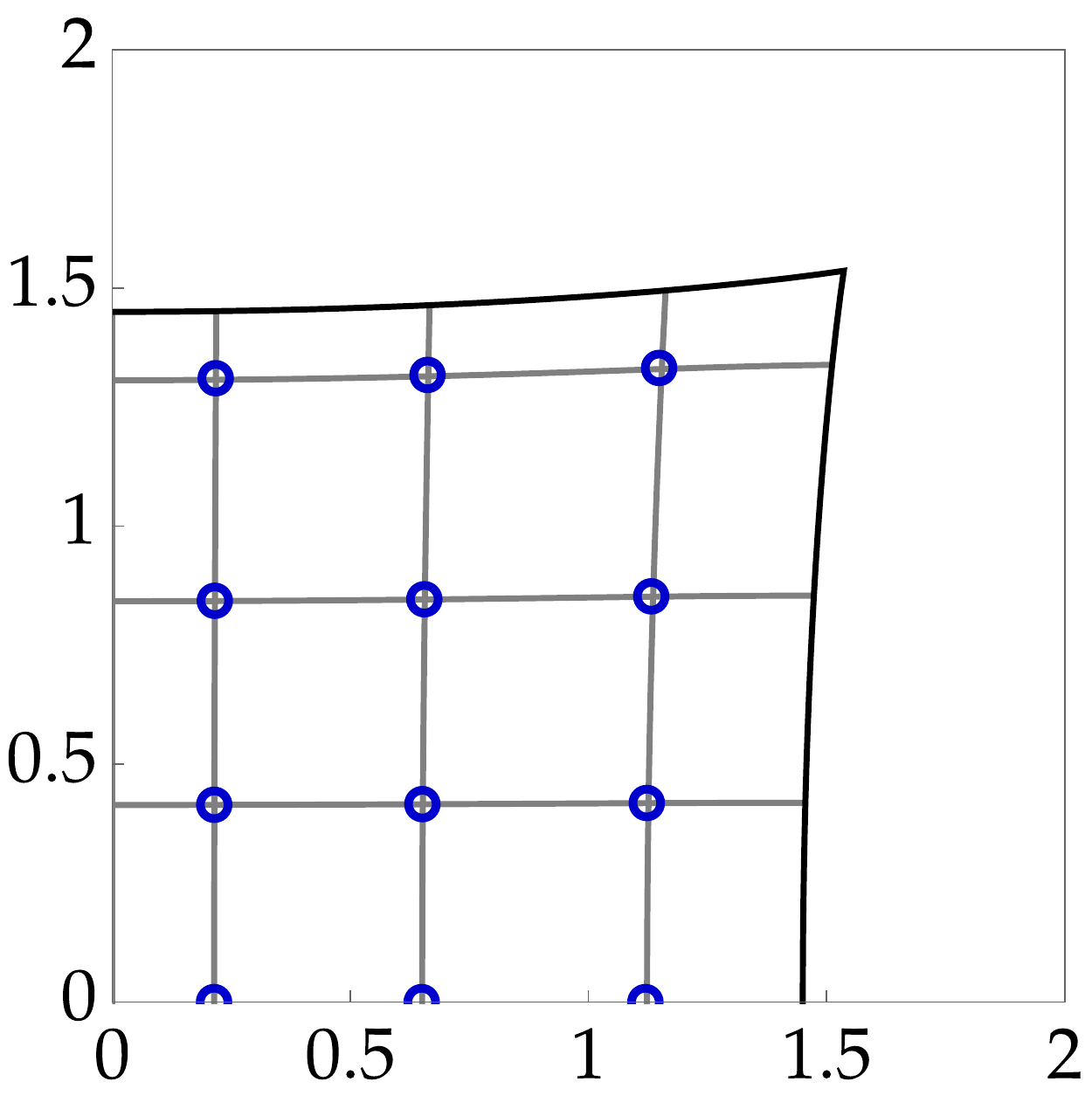}\\
\includegraphics[height=1.5 in]{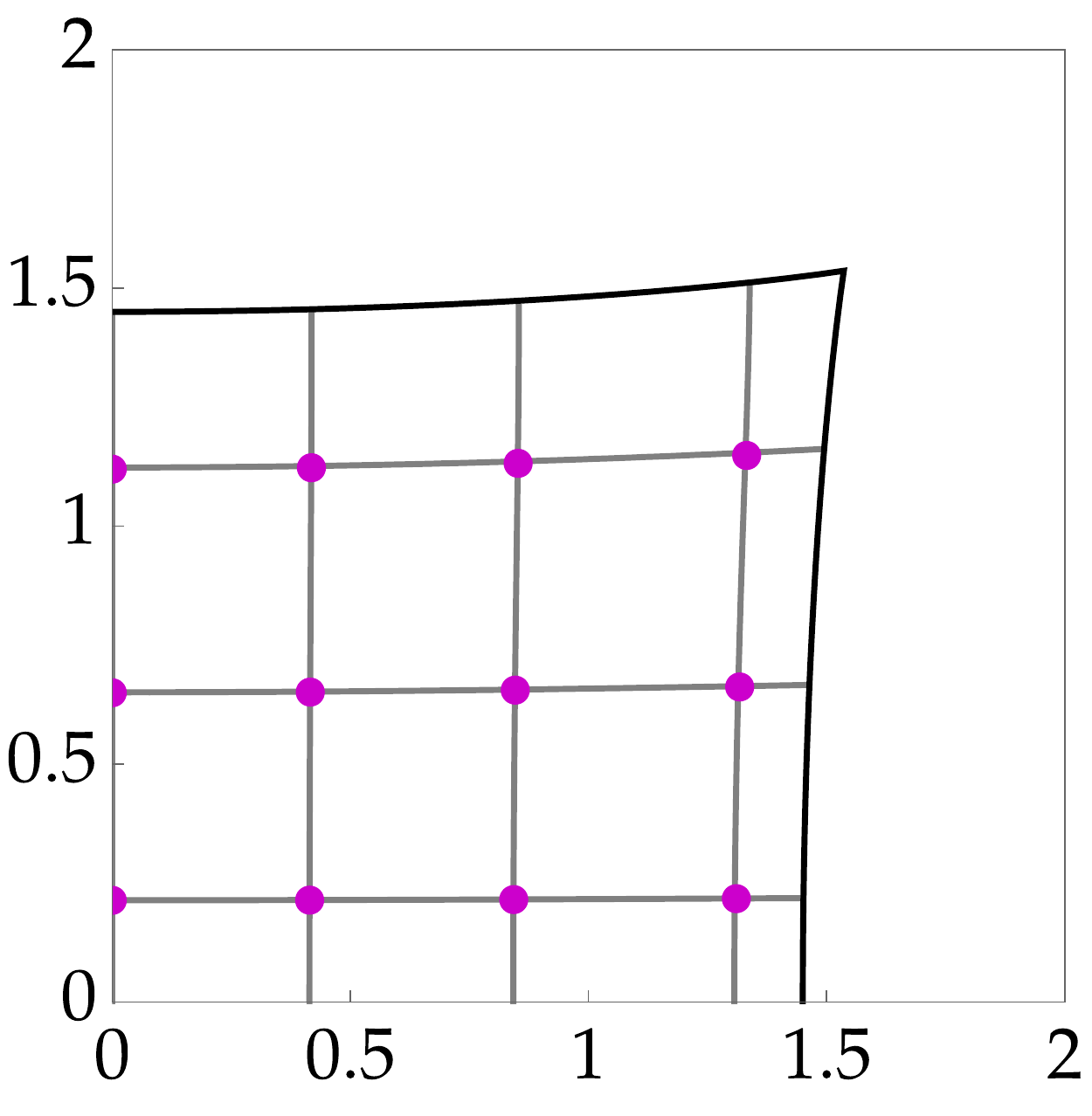}\\
\includegraphics[height=1.5 in]{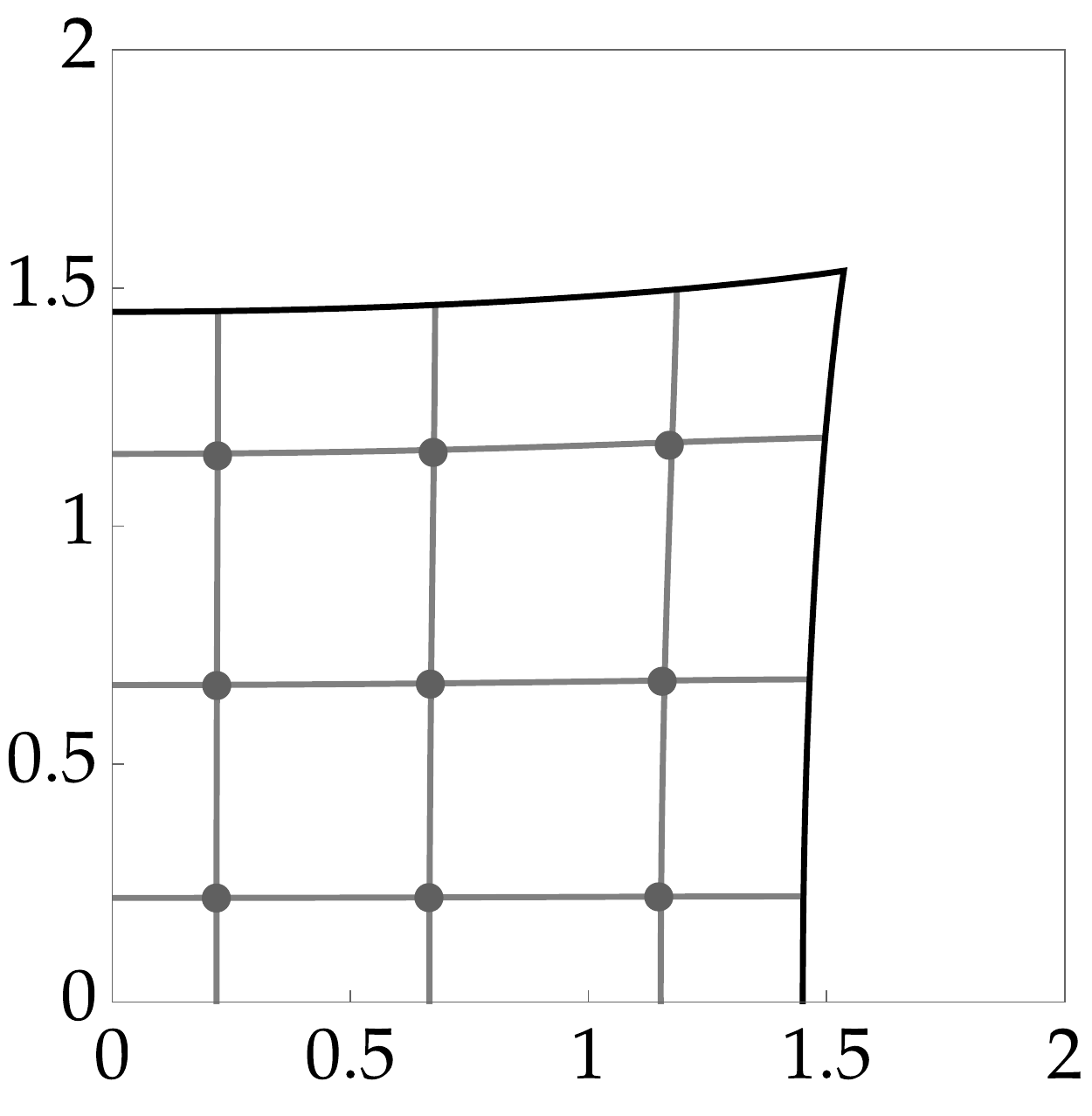}
\end{center}
\caption{$U=T^{-\frac{1}{2}}u^{[1]}_\mathrm{gH}(x;6,6)$; \\$\rho=1$; $\kappa=-\tfrac{10}{3}$.}
\end{subfigure}%
\begin{subfigure}{1.5 in}
\begin{center}
\includegraphics[height=1.5 in]{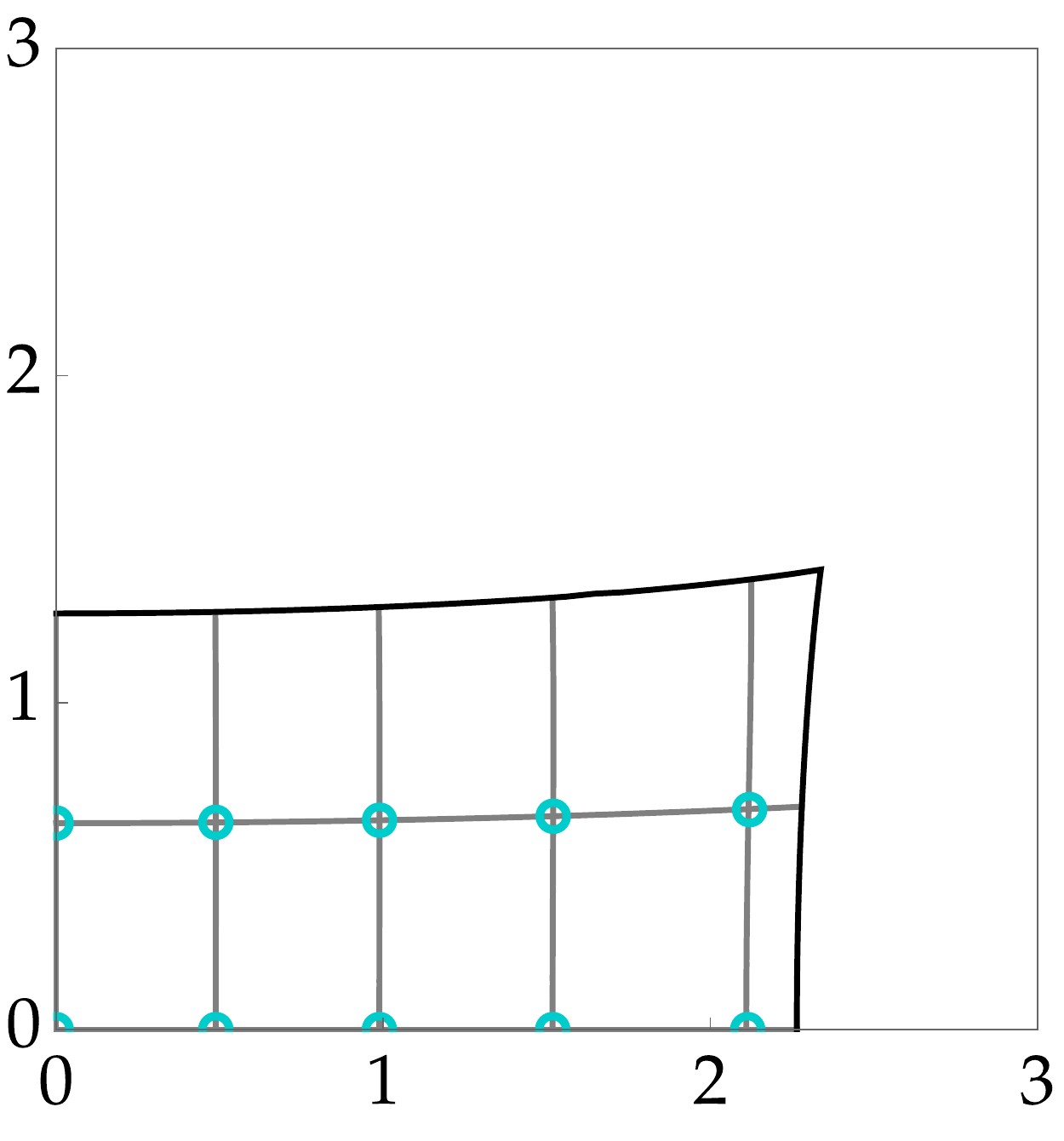}\\
\includegraphics[height=1.5 in]{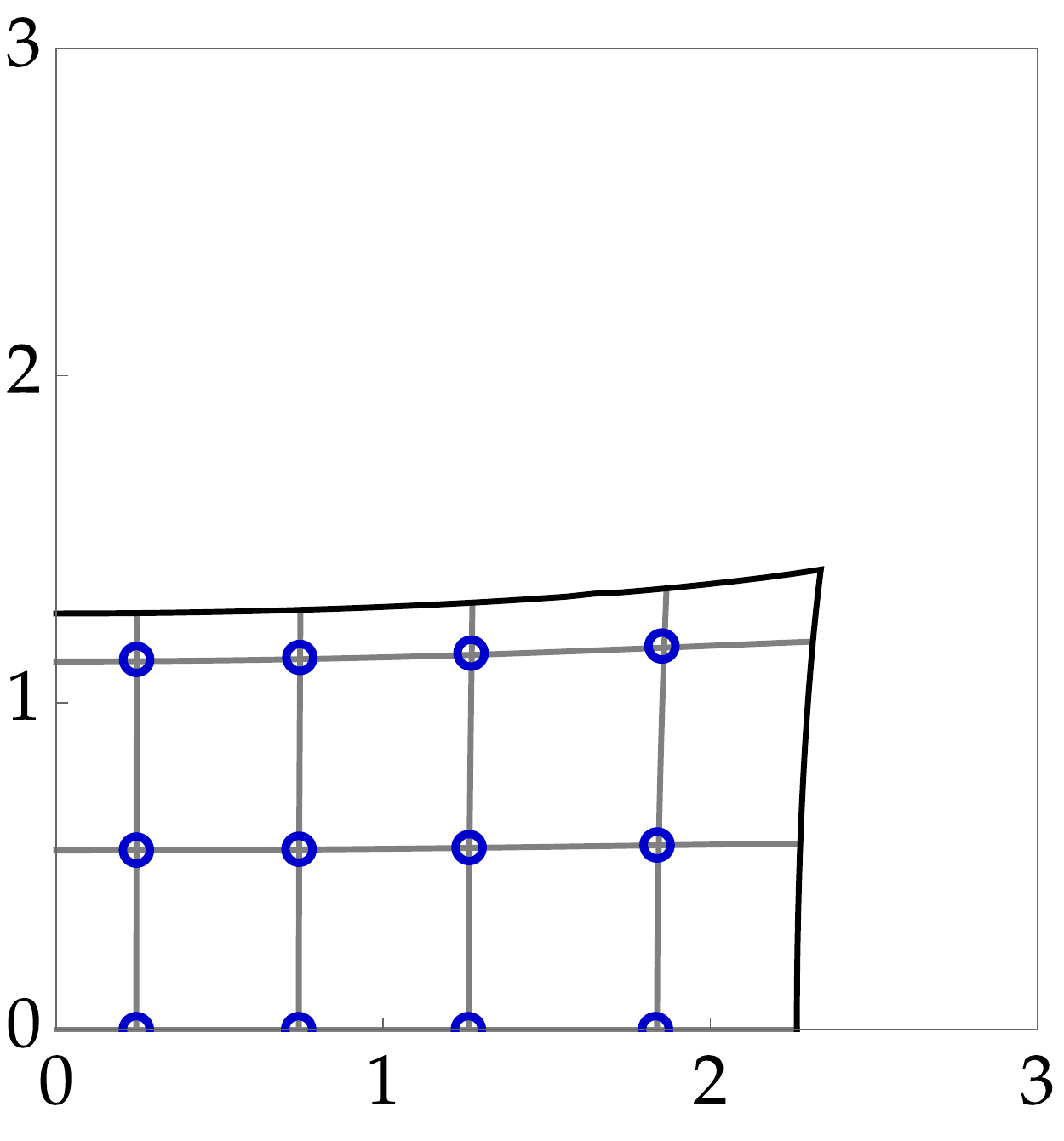}\\
\includegraphics[height=1.5 in]{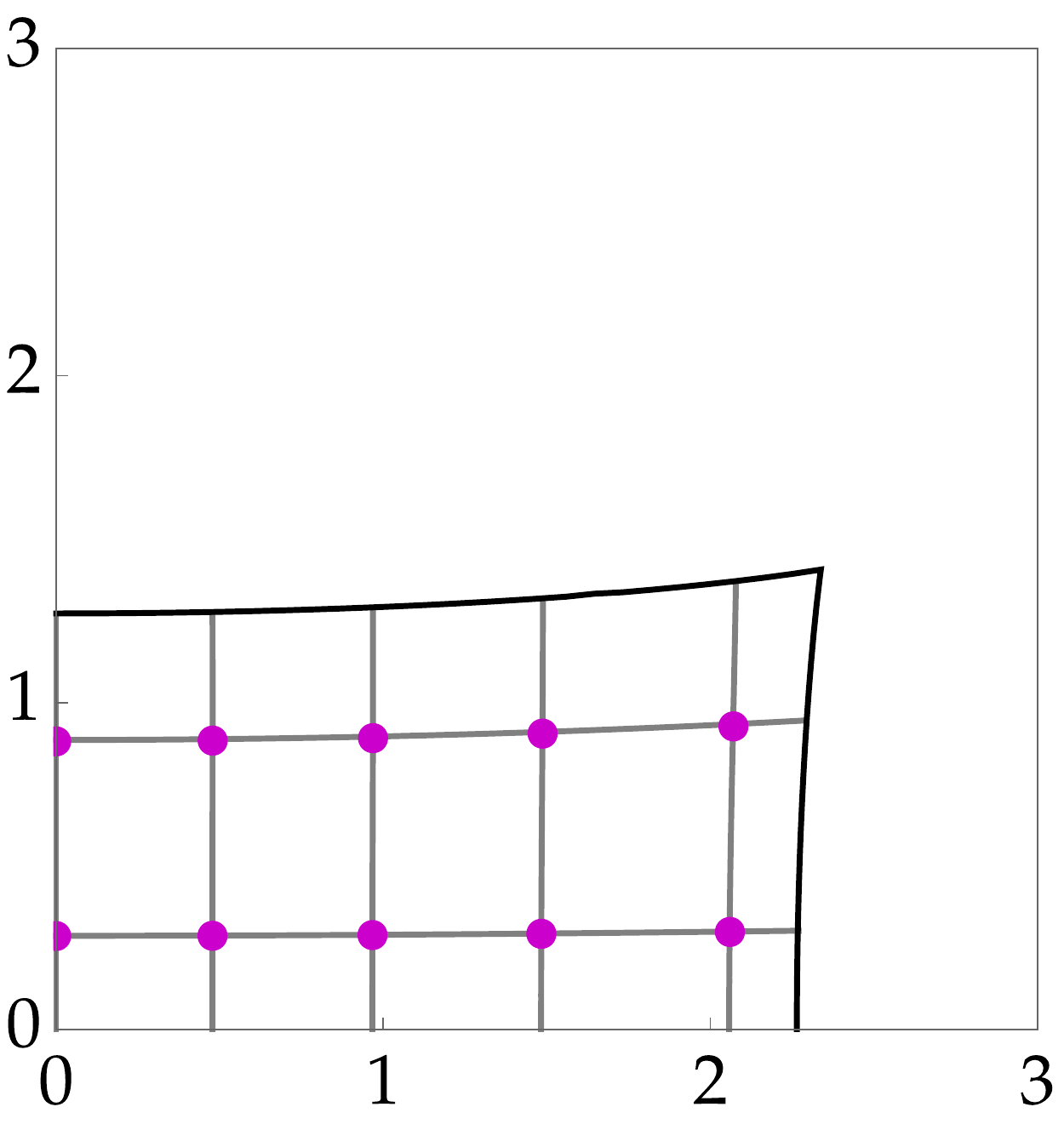}\\
\includegraphics[height=1.5 in]{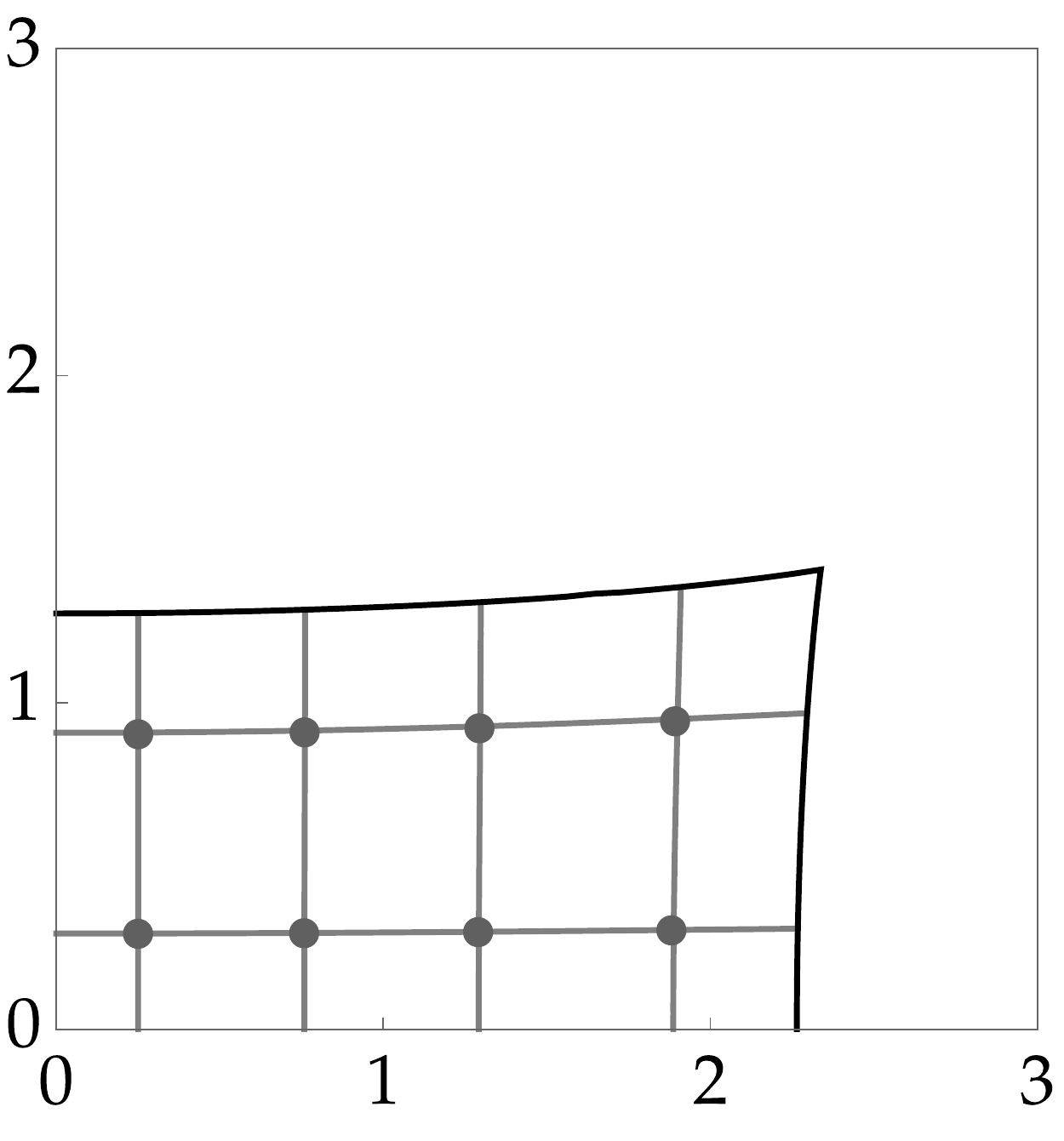}
\end{center}
\caption{$U=T^{-\frac{1}{2}}u^{[1]}_\mathrm{gH}(x;8,4)$; \\$\rho=\tfrac{1}{2}$; $\kappa=-\tfrac{11}{2}$.}
\end{subfigure}%
\begin{subfigure}{1.5 in}
\begin{center}
\includegraphics[height=1.5 in]{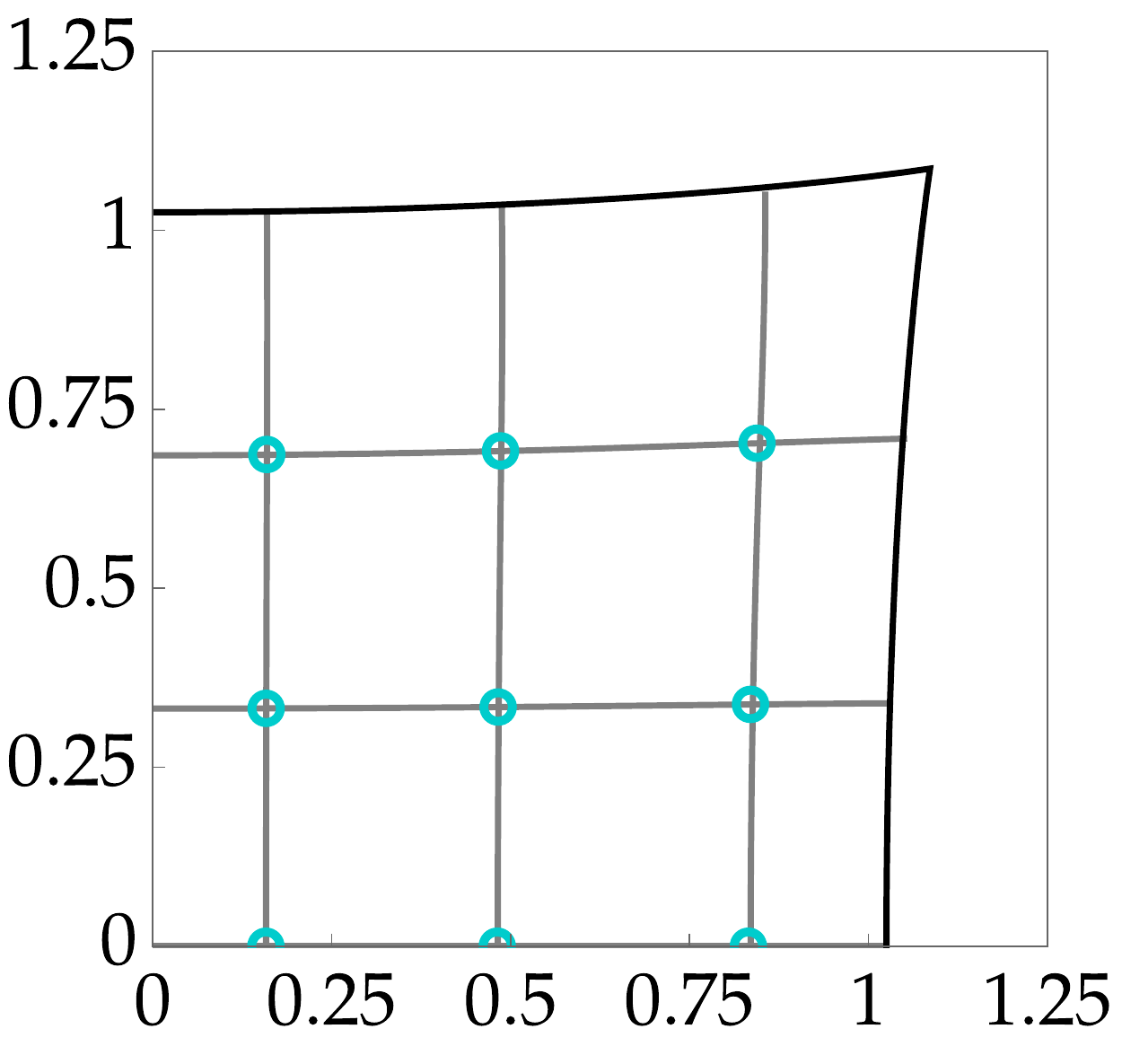}\\
\includegraphics[height=1.5 in]{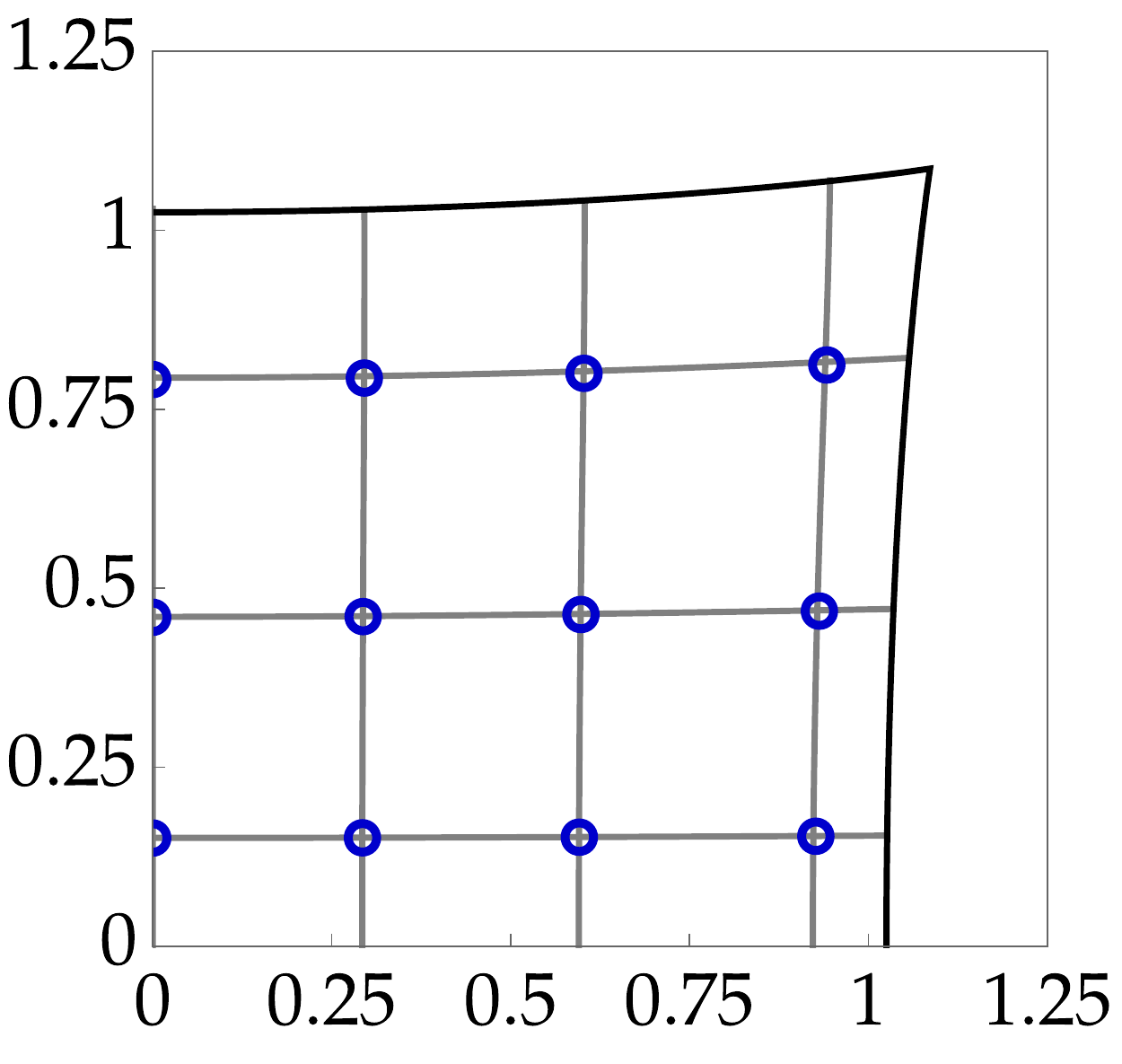}\\
\includegraphics[height=1.5 in]{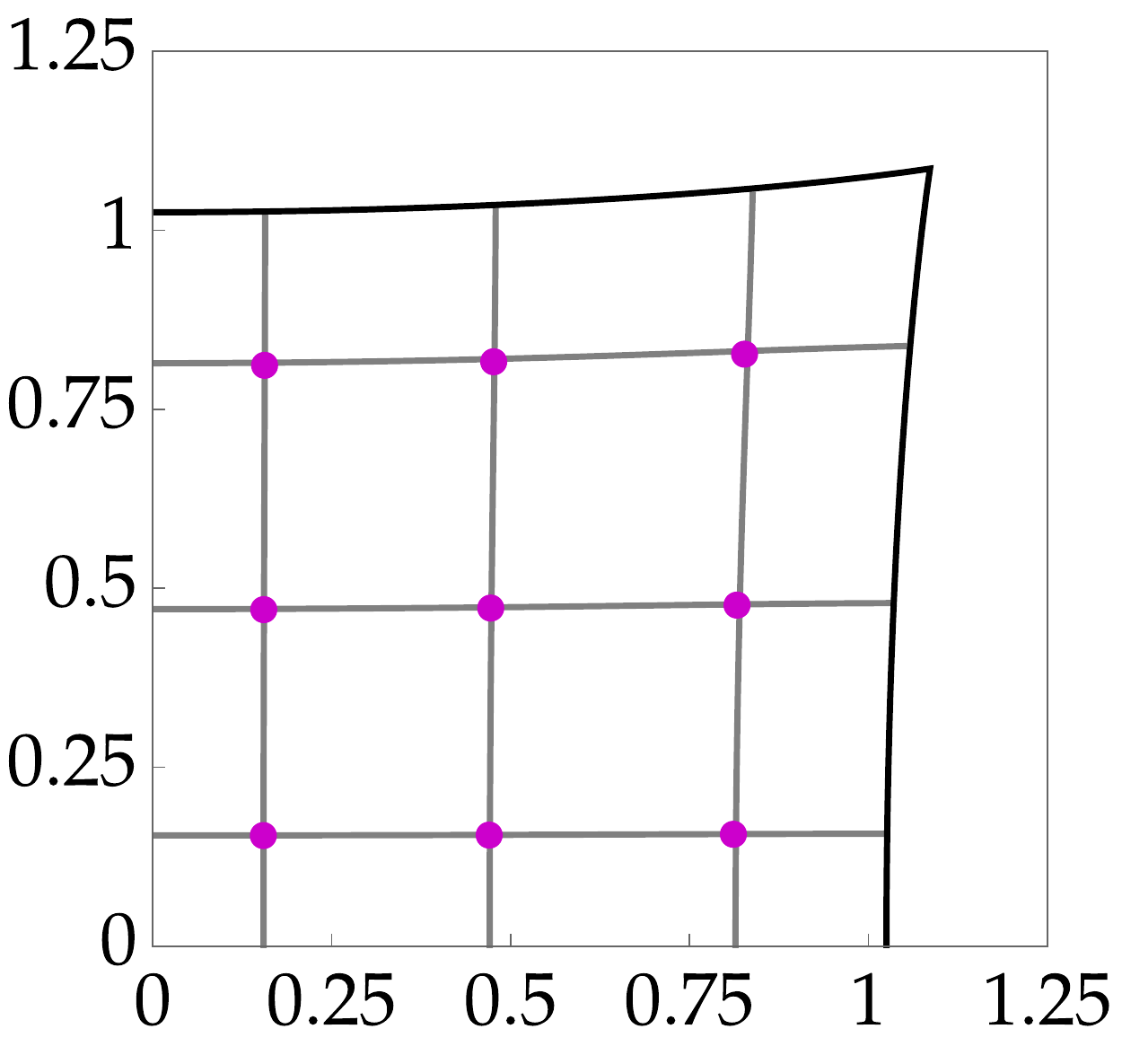}\\
\includegraphics[height=1.5 in]{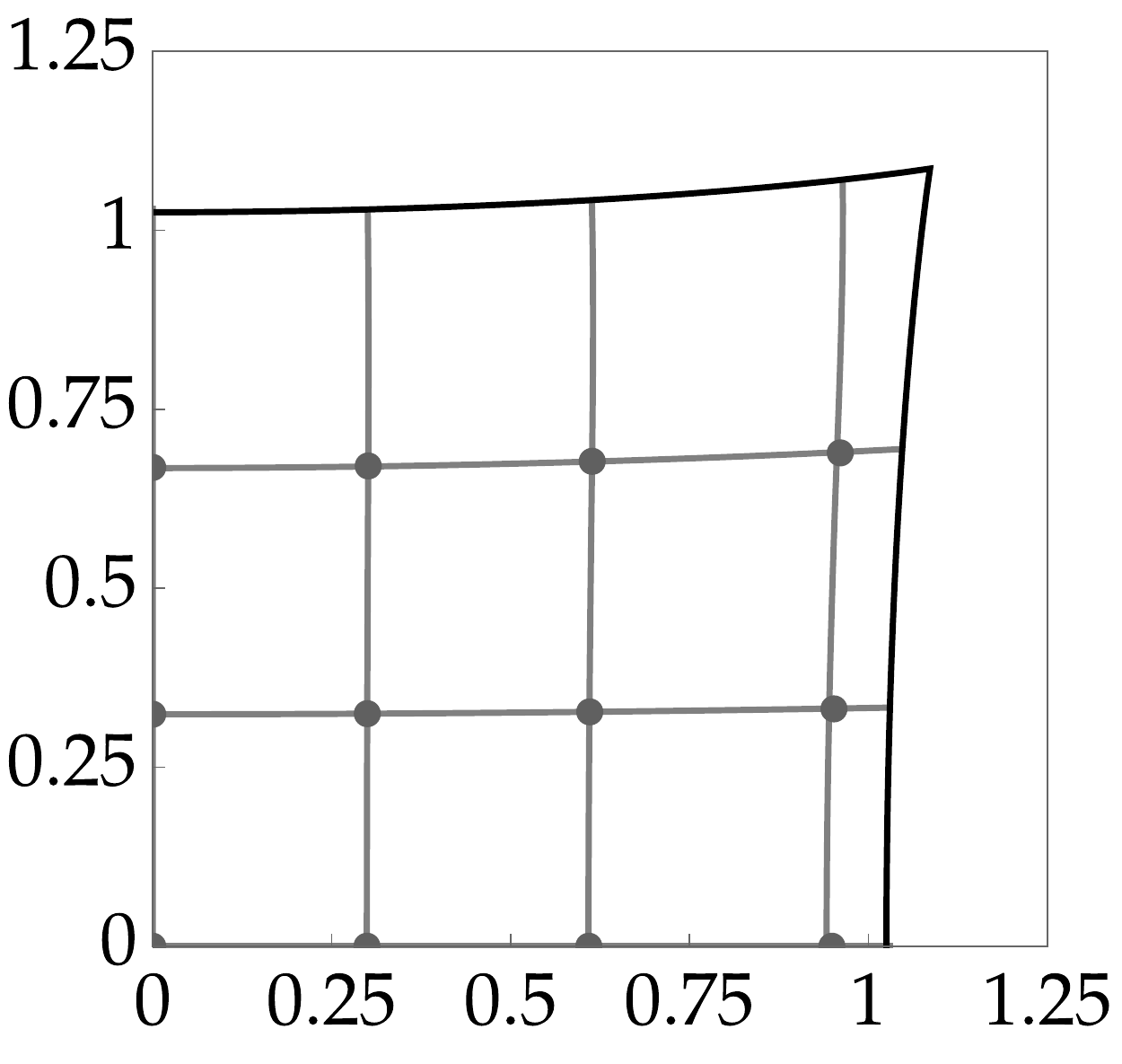}
\end{center}
\caption{$U=T^{-\frac{1}{2}}u^{[3]}_\mathrm{gH}(x;6,5)$; \\$\rho=\tfrac{5}{6}$; $\kappa=0$.}
\end{subfigure}%
\begin{subfigure}{1.5 in}
\begin{center}
\includegraphics[height=1.5 in]{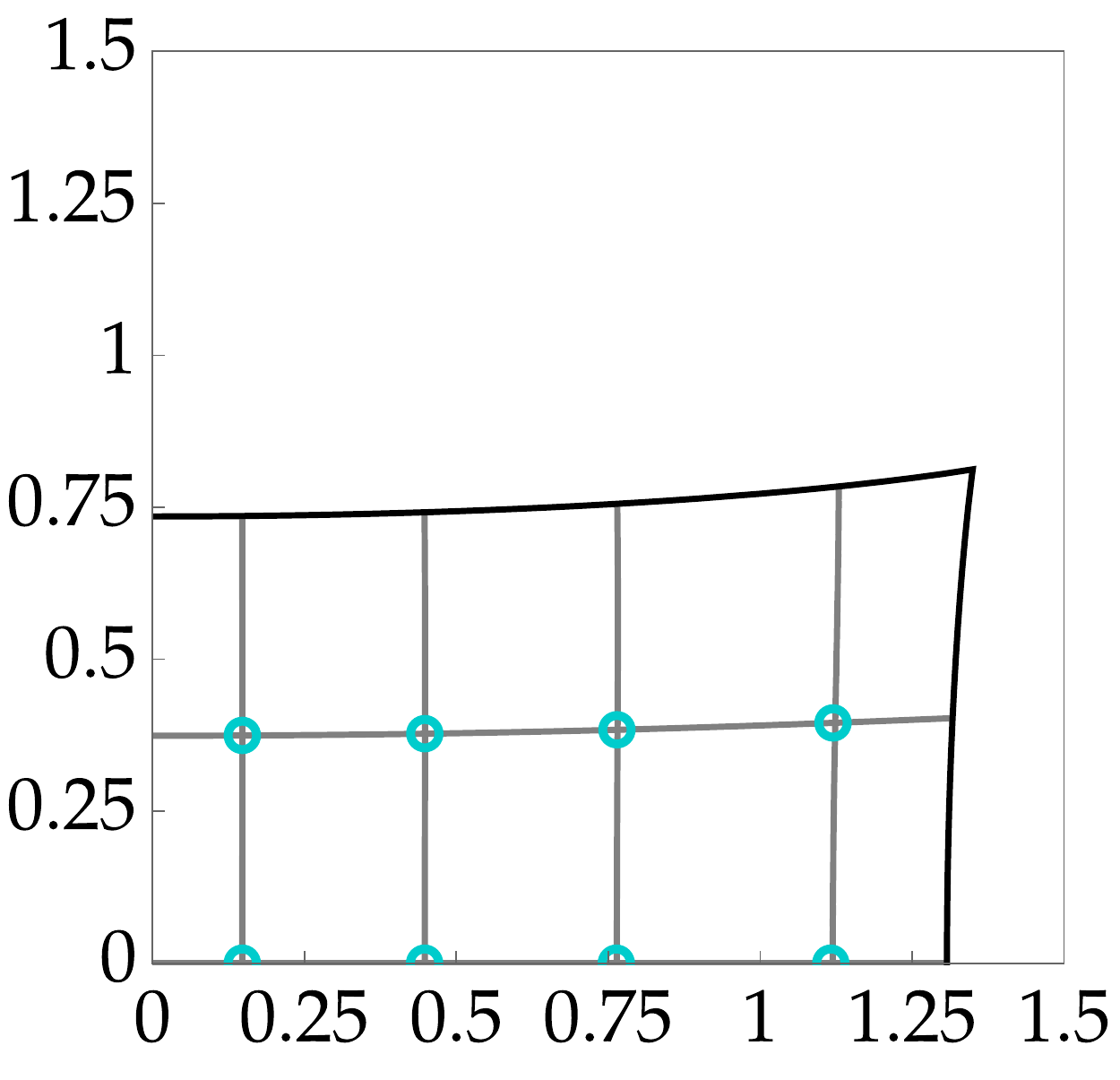}\\
\includegraphics[height=1.5 in]{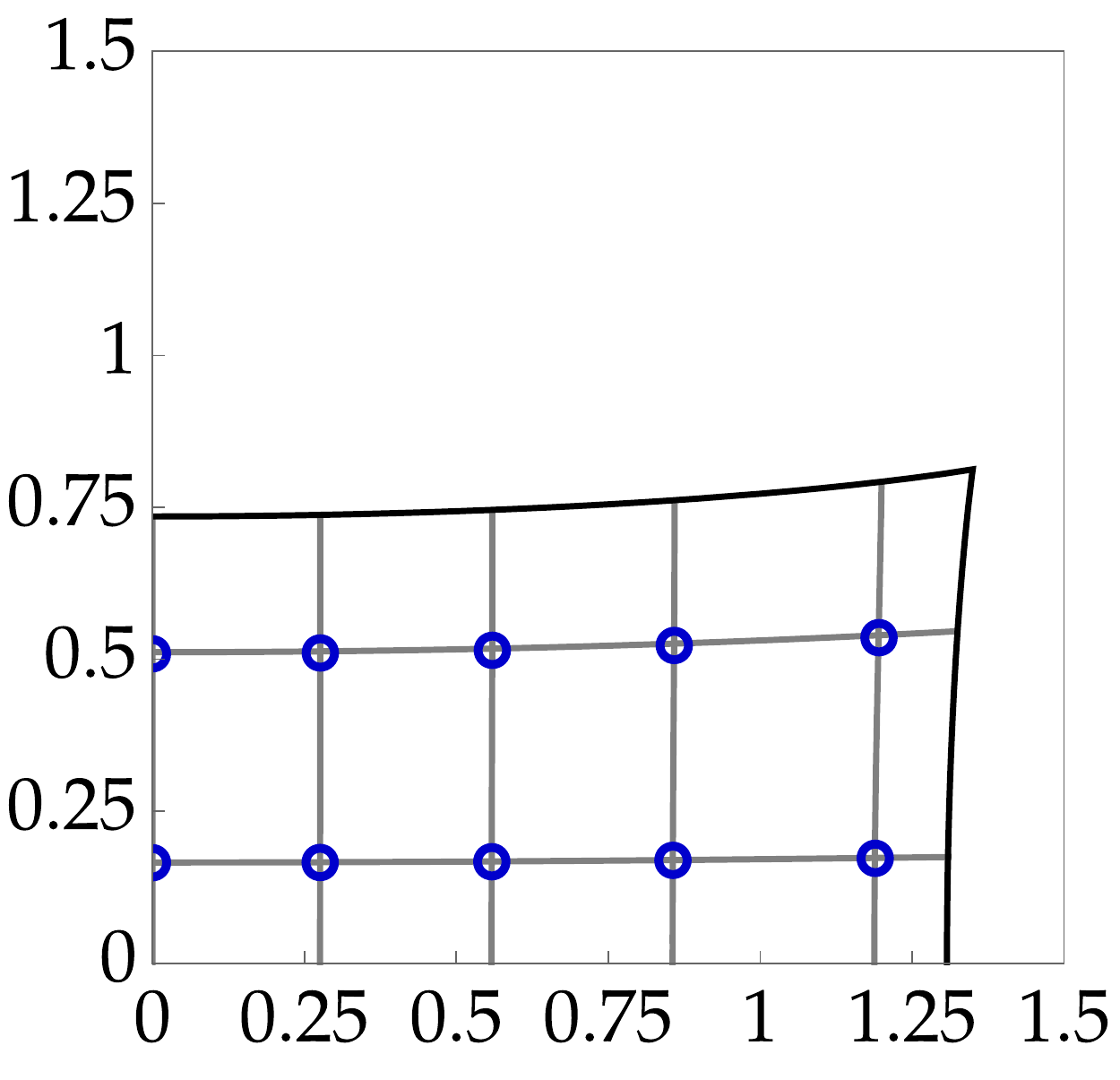}\\
\includegraphics[height=1.5 in]{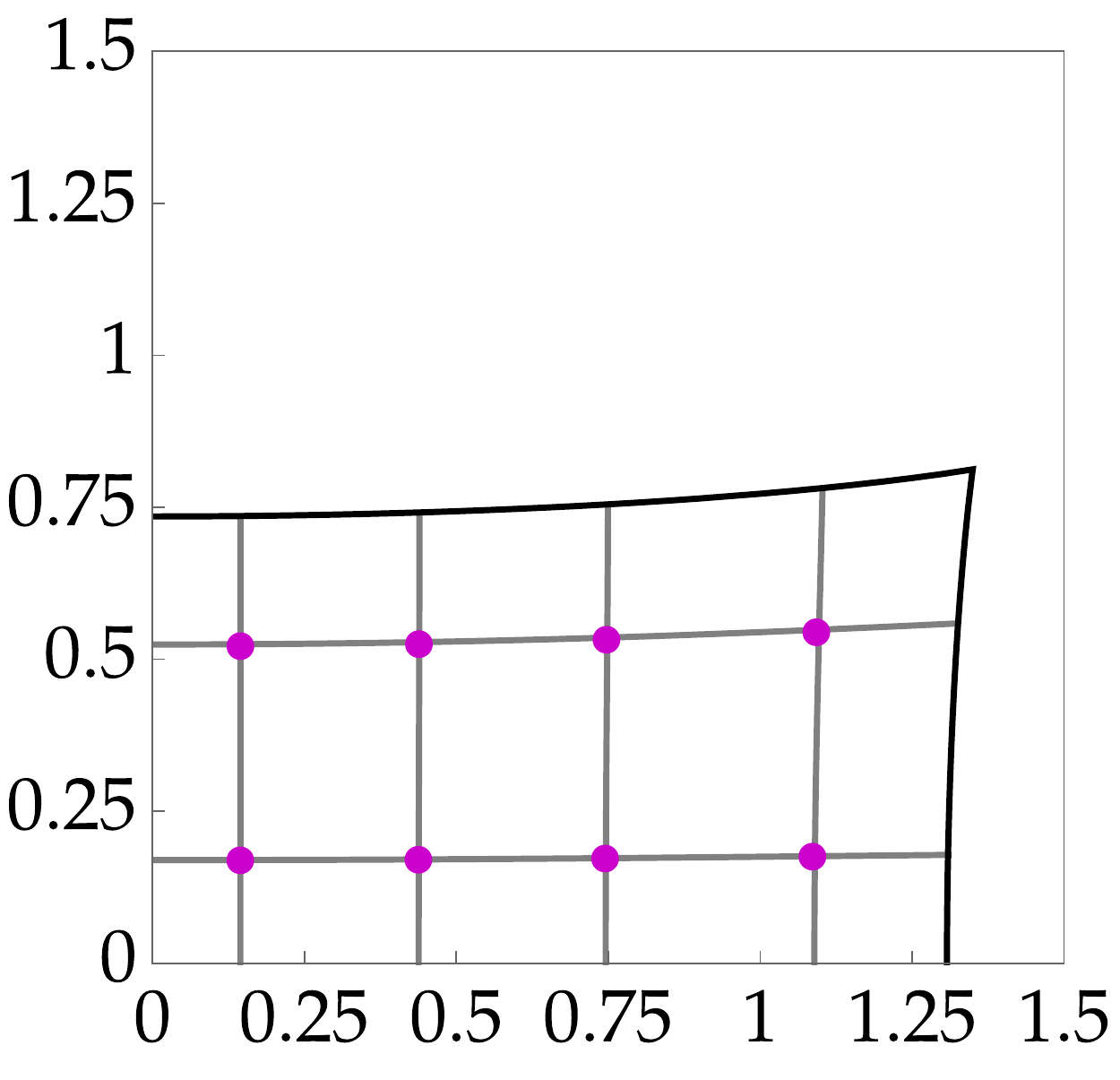}\\
\includegraphics[height=1.5 in]{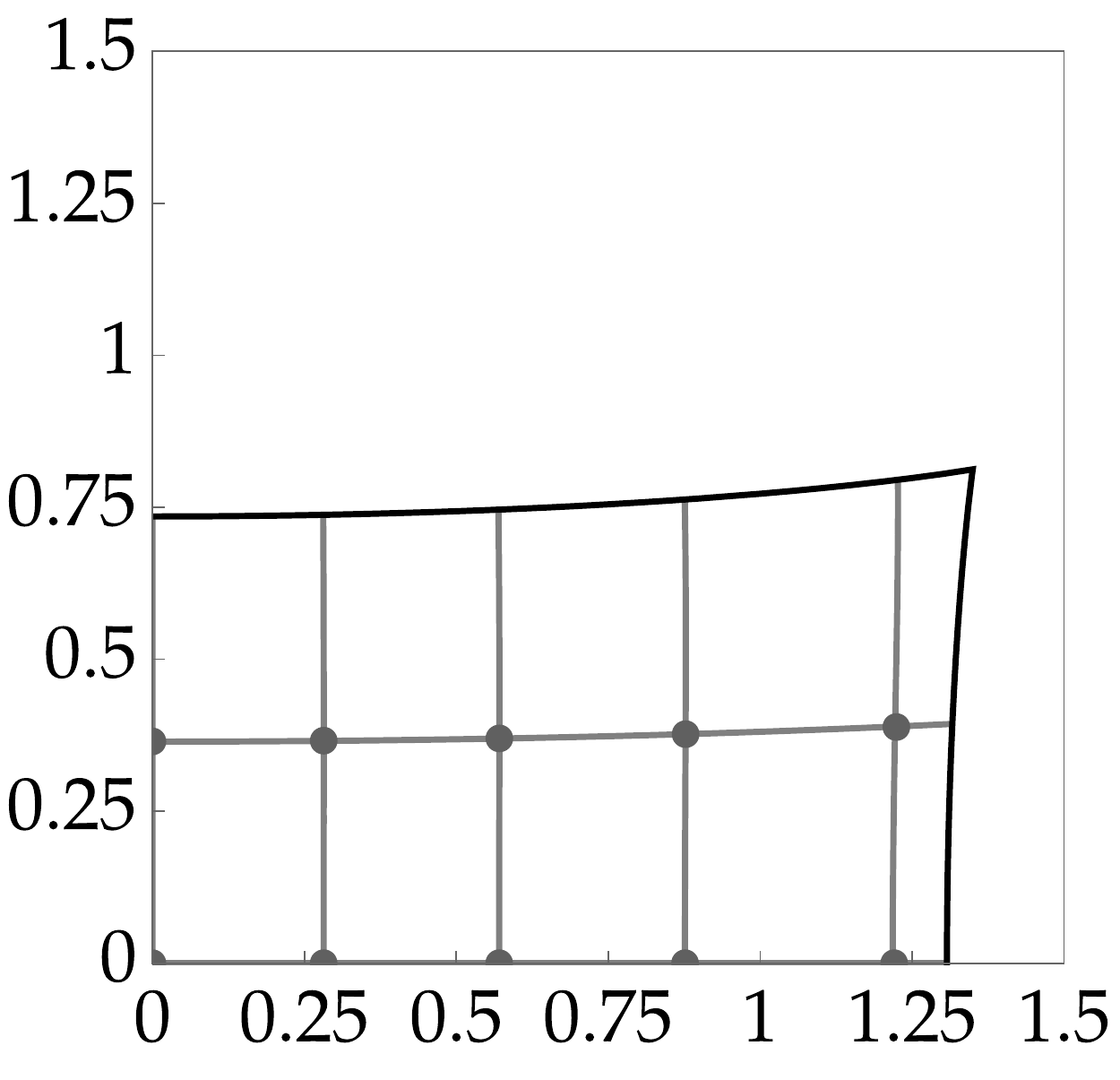}
\end{center}
\caption{$U=T^{-\frac{1}{2}}u^{[3]}_\mathrm{gH}(x;8,3)$; \\$\rho=\frac{3}{8}$; $\kappa=\tfrac{1}{3}$.}
\end{subfigure}
\end{center}
\caption{Quantitative comparison of zeros (circles, cyan for positive derivative and blue for negative derivative, top two rows of plots) and poles (dots, magenta for positive residue and gray for negative residue, bottom two rows of plots) of scaled gH rational solutions $U$ with the approximations given according to Corollary~\ref{cor:poles-and-zeros} by the intersection points of integer level curves (gray) of the two conditions in \eqref{eq:intro-zeros-quantize} or \eqref{eq:intro-poles-quantize} plotted in the $y=T^{-\frac{1}{2}}x$-plane.}
\label{fig:gH-poles-zeros}
\end{figure}
\begin{figure}[h]
\begin{center}
\begin{subfigure}{1.5 in}
\begin{center}
\includegraphics[height=1.5 in]{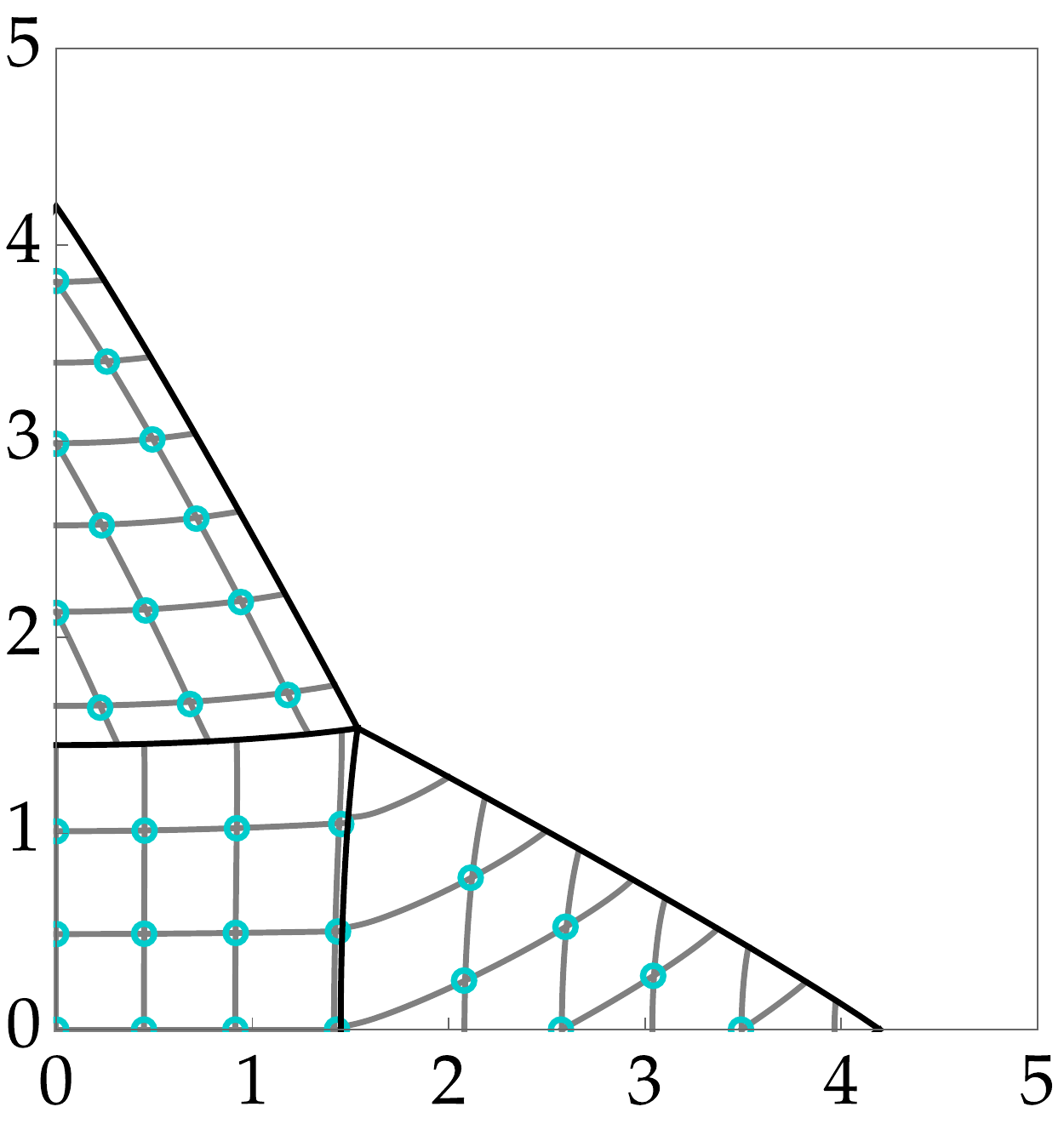}\\
\includegraphics[height=1.5 in]{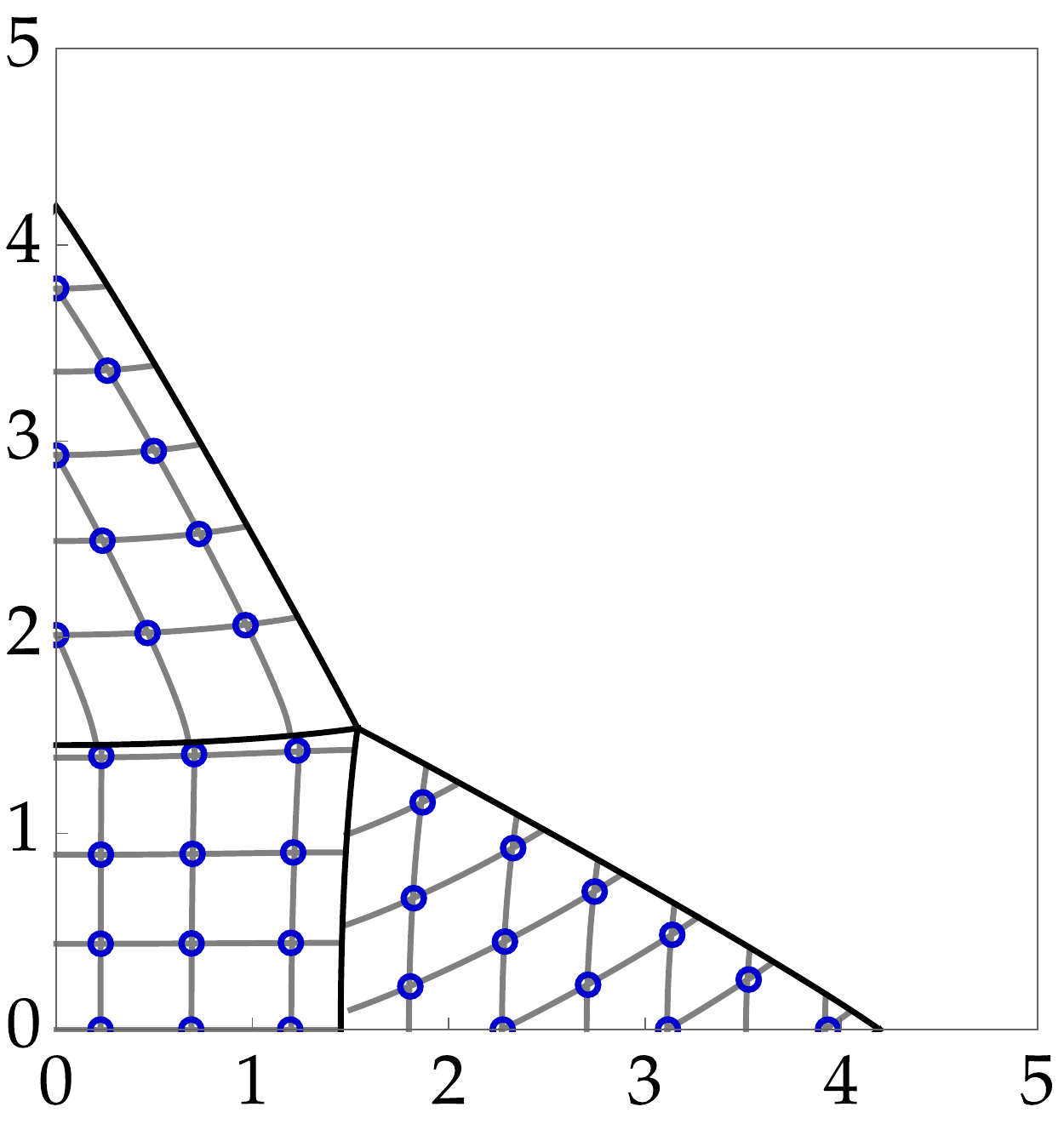}\\
\includegraphics[height=1.5 in]{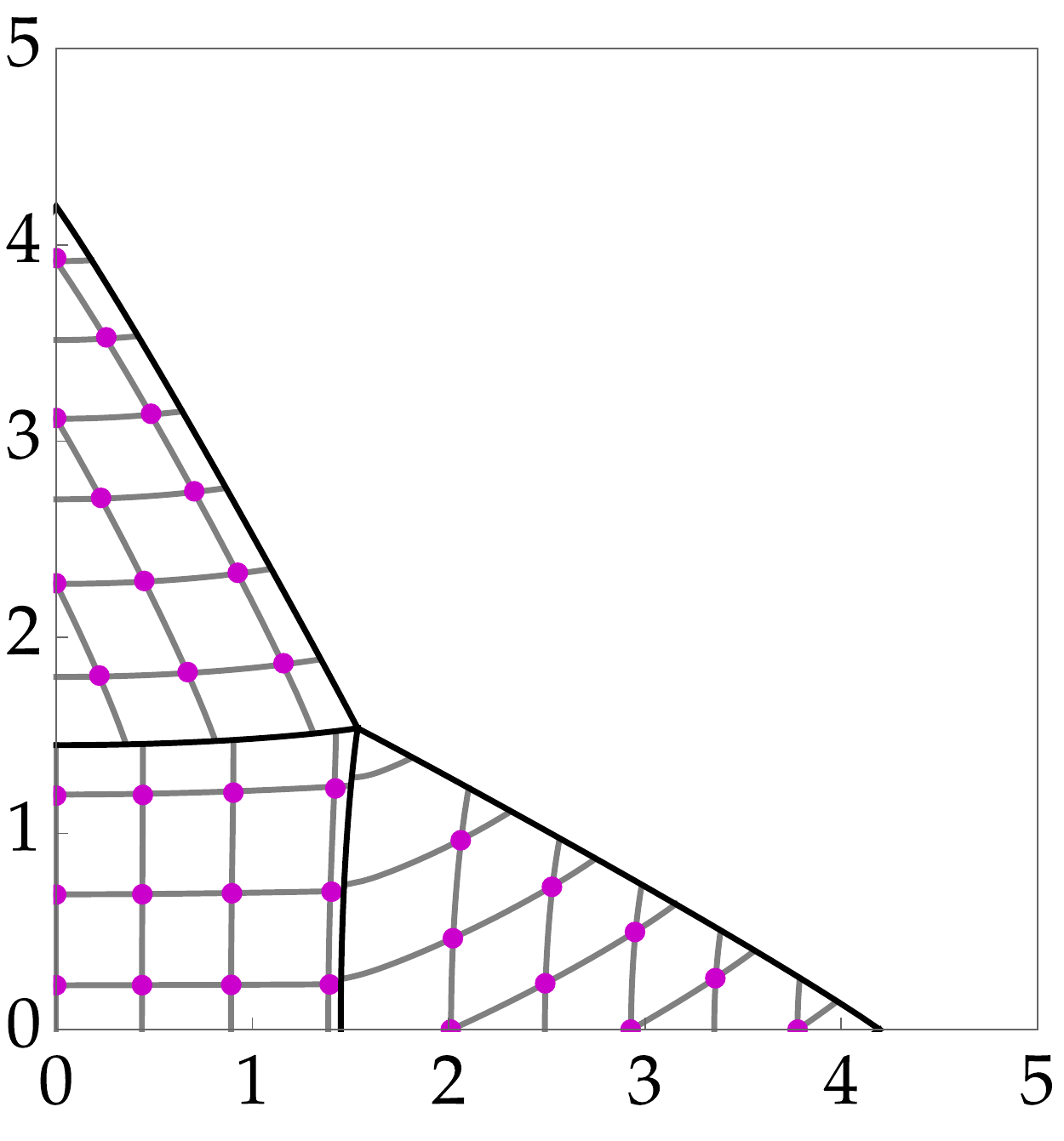}\\
\includegraphics[height=1.5 in]{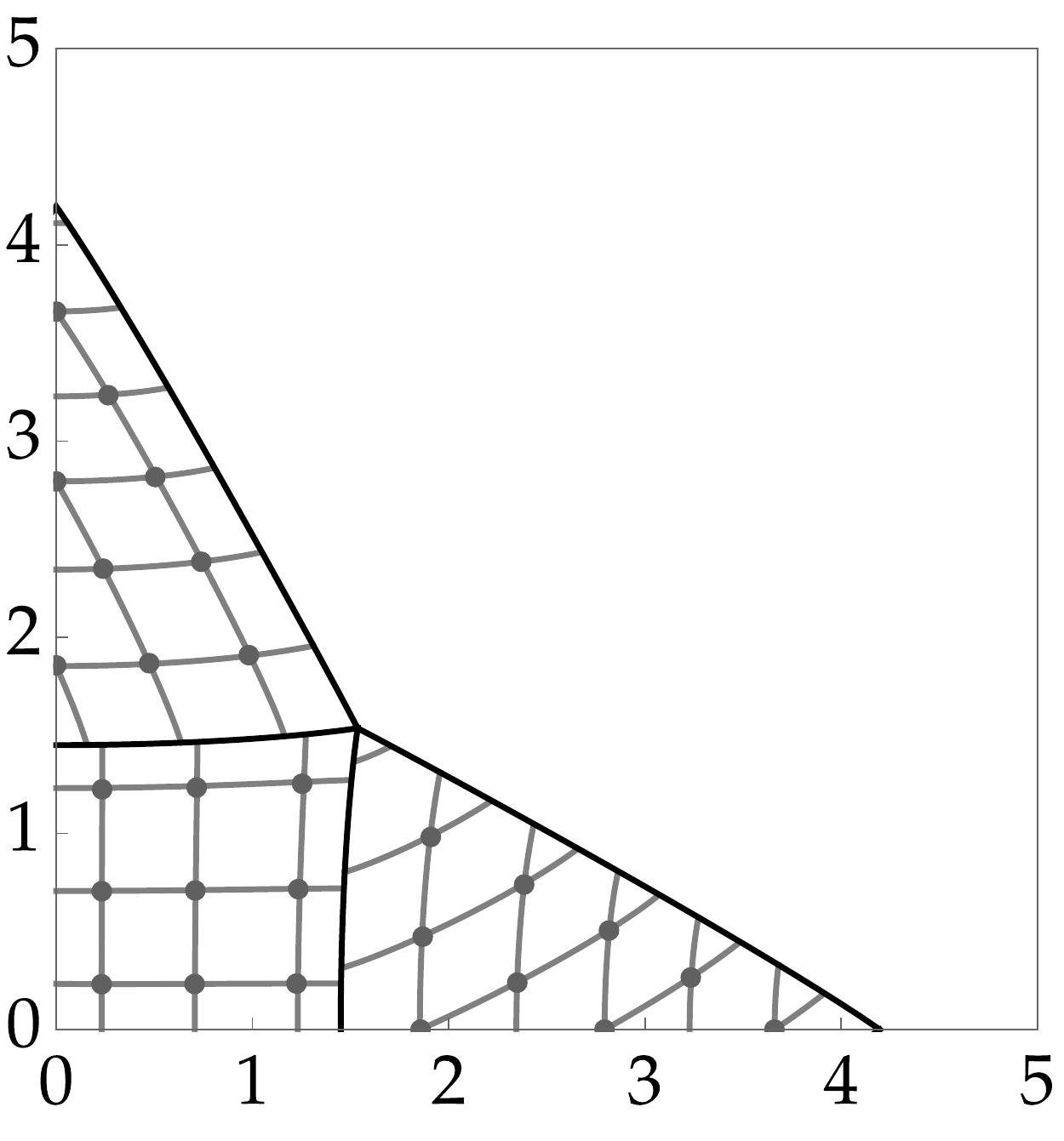}
\end{center}
\caption{$U=T^{-\frac{1}{2}}u^{[1]}_\mathrm{gO}(x;6,6)$; \\$\rho=1$; $\kappa=-\tfrac{57}{17}$.}
\end{subfigure}%
\begin{subfigure}{1.5 in}
\begin{center}
\includegraphics[height=1.5 in]{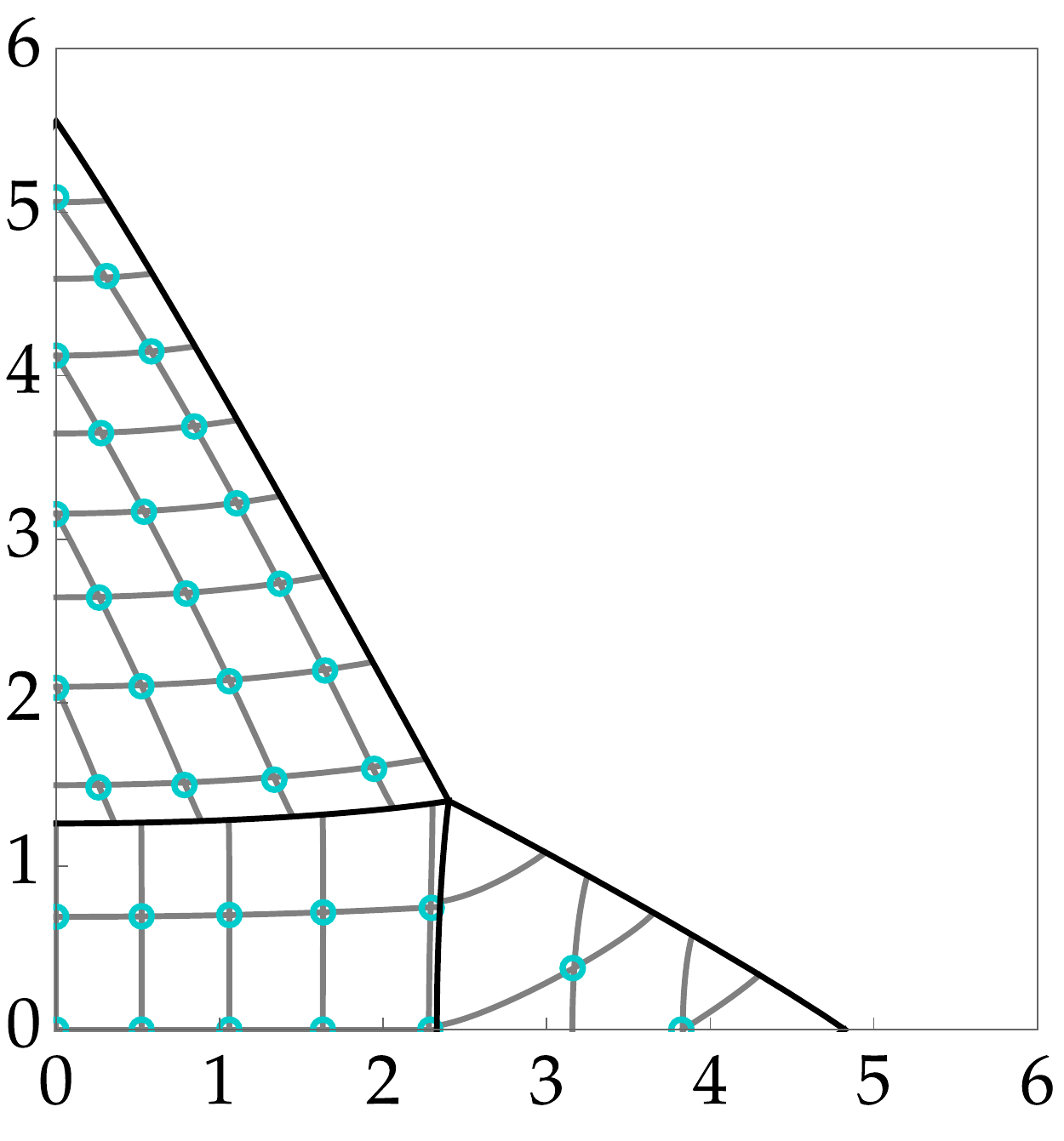}\\
\includegraphics[height=1.5 in]{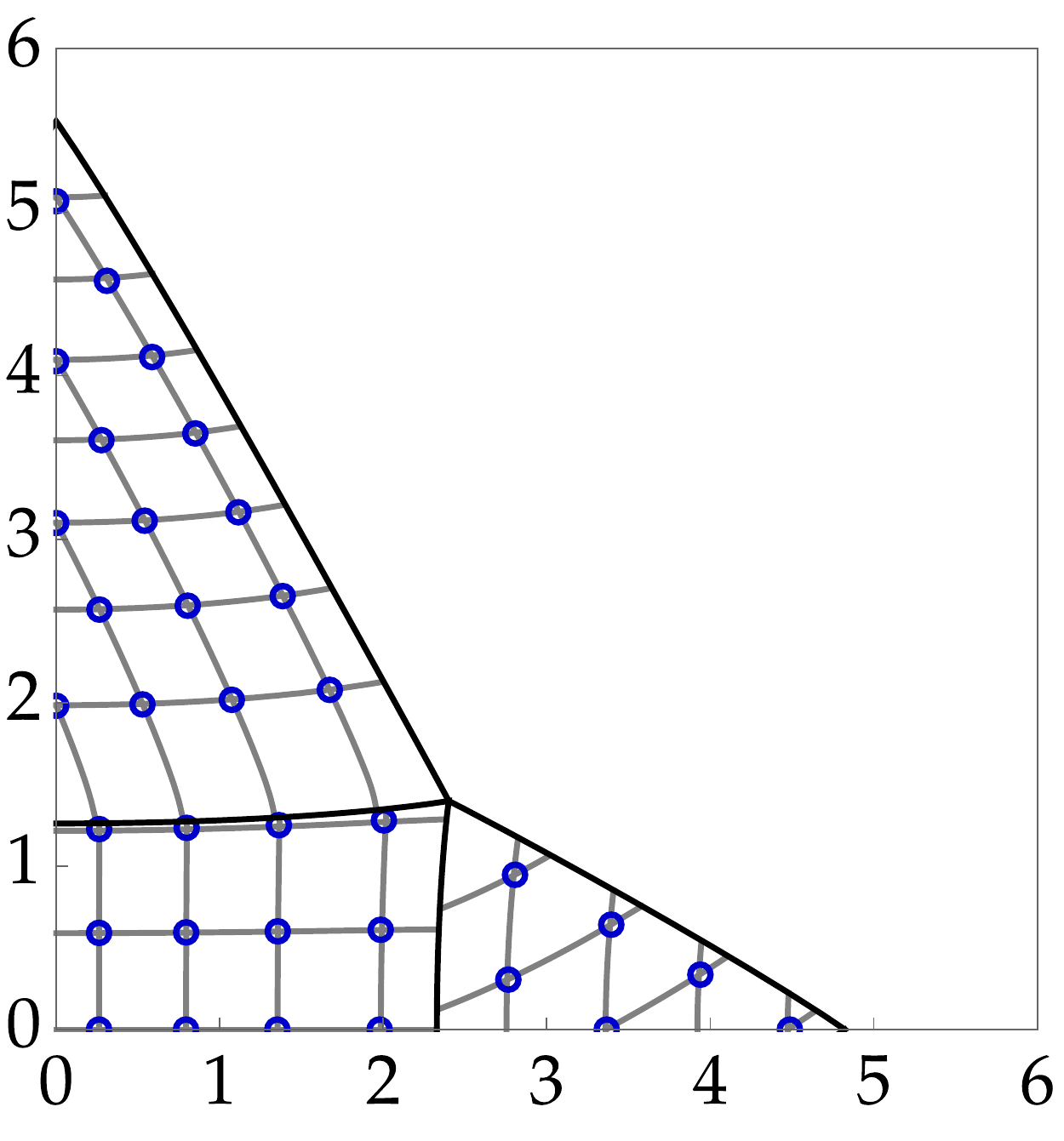}\\
\includegraphics[height=1.5 in]{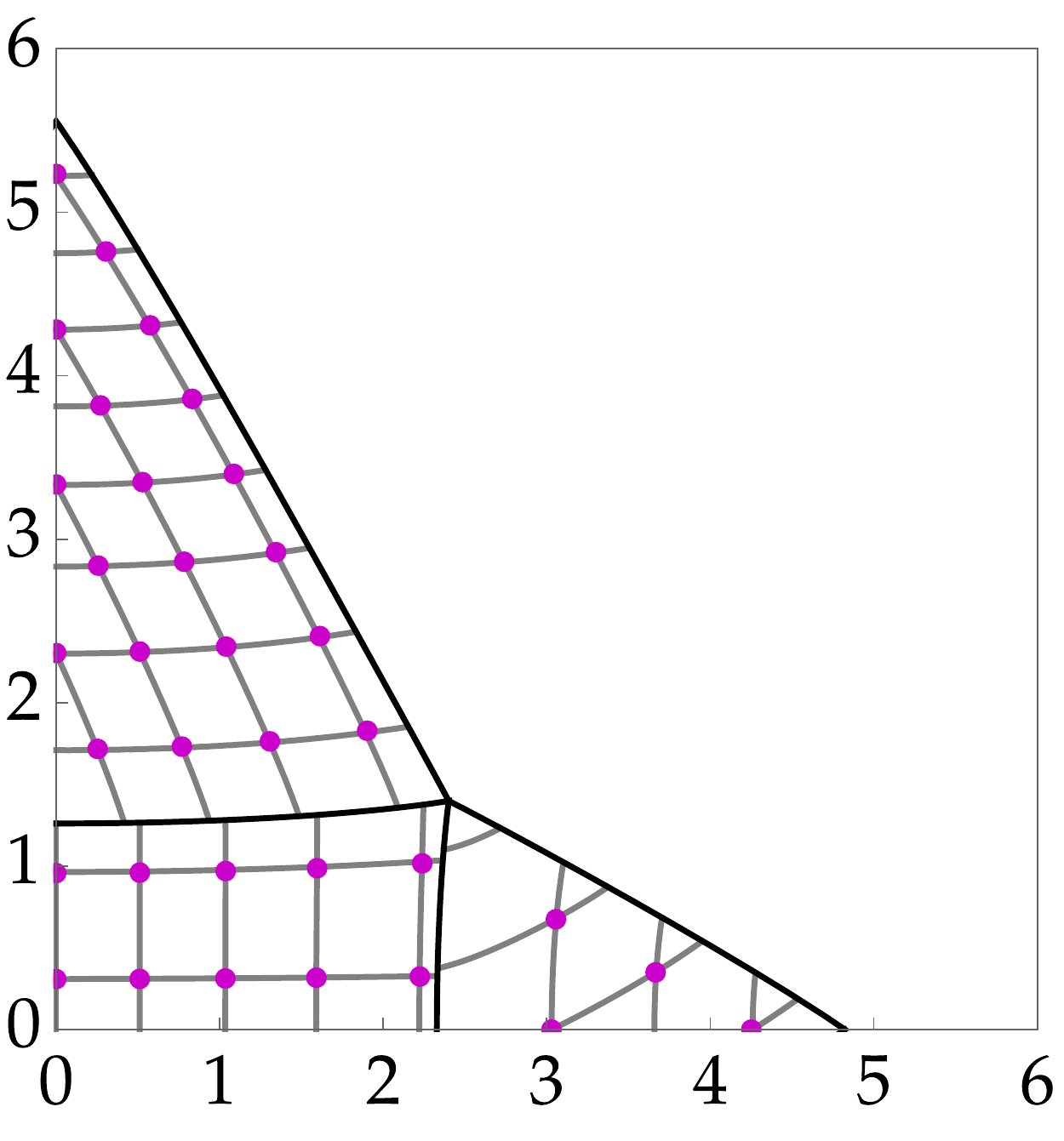}\\
\includegraphics[height=1.5 in]{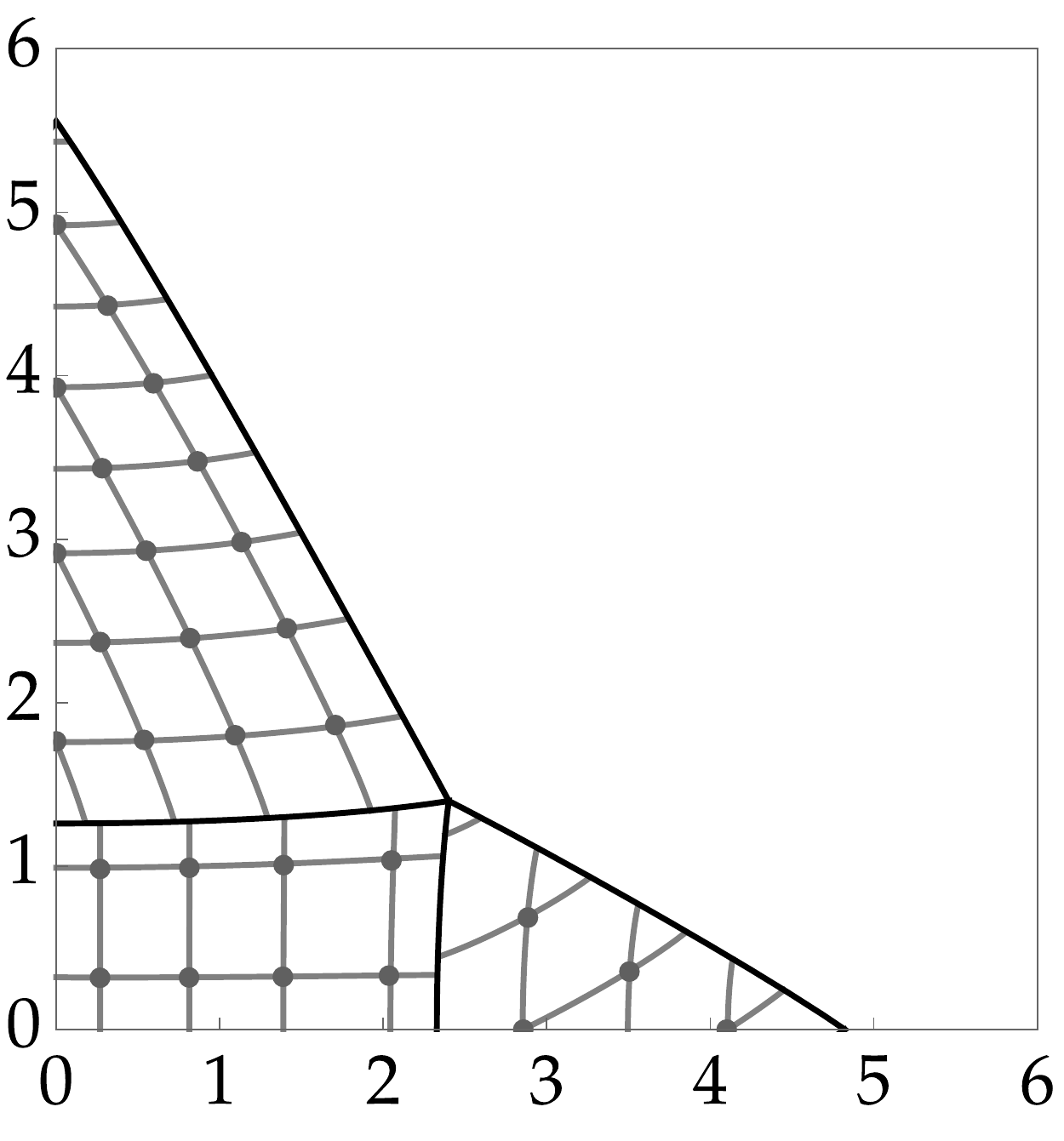}
\end{center}
\caption{$U=T^{-\frac{1}{2}}u^{[1]}_\mathrm{gO}(x;8,4)$; \\$\rho=\tfrac{1}{2}$; $\kappa=-\tfrac{63}{11}$.}
\end{subfigure}%
\begin{subfigure}{1.5 in}
\begin{center}
\includegraphics[height=1.5 in]{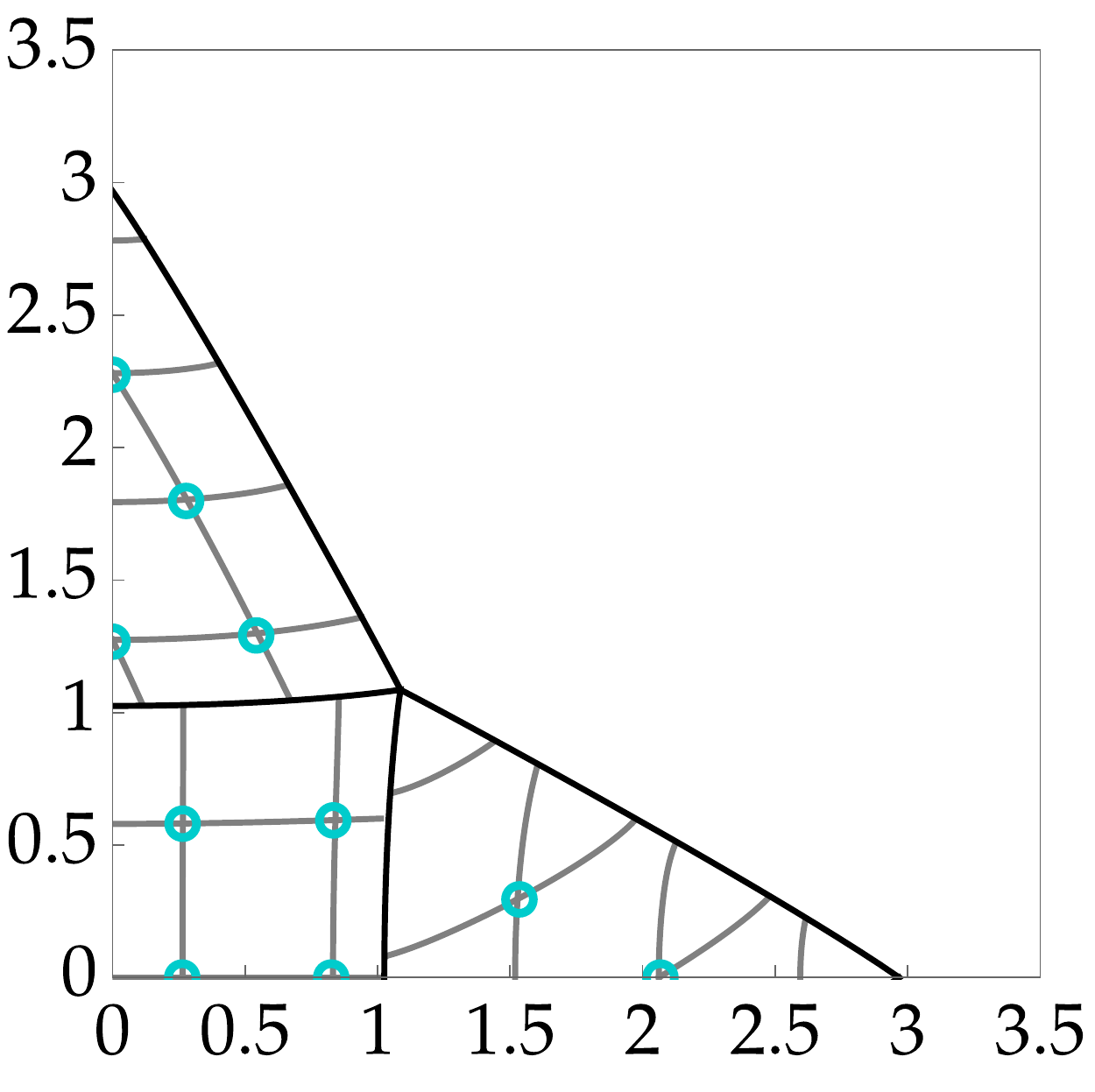}\\
\includegraphics[height=1.5 in]{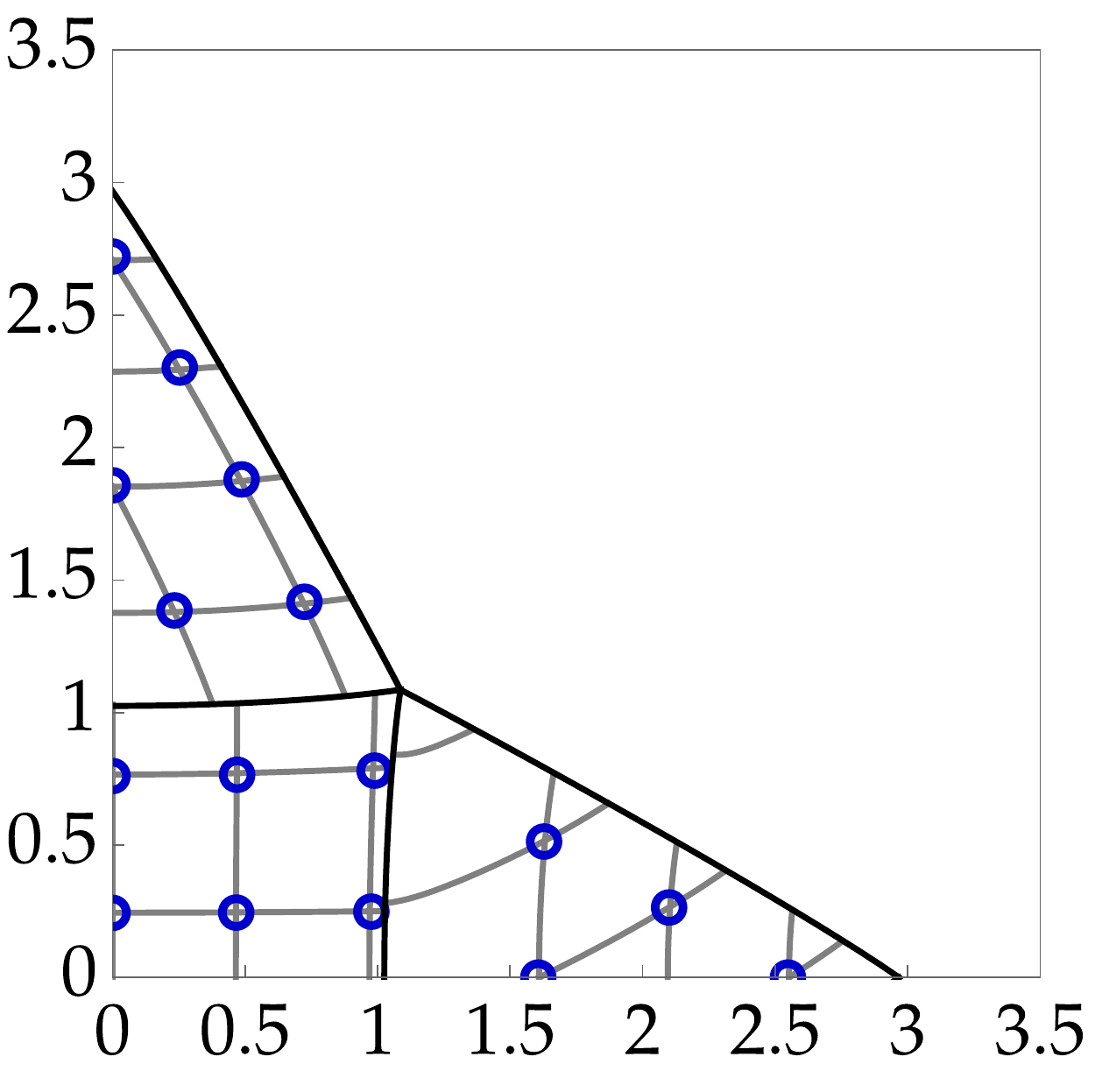}\\
\includegraphics[height=1.5 in]{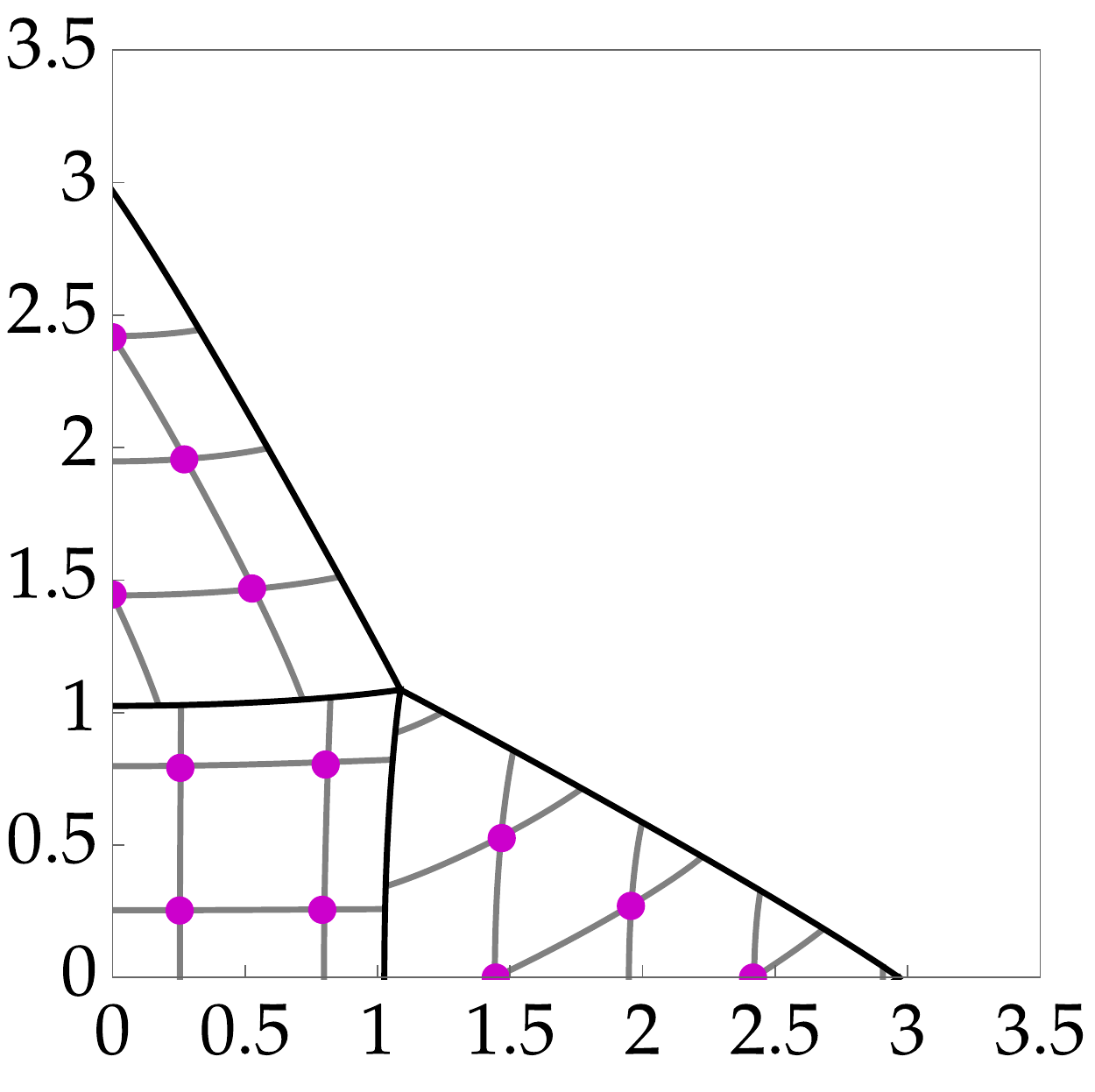}\\
\includegraphics[height=1.5 in]{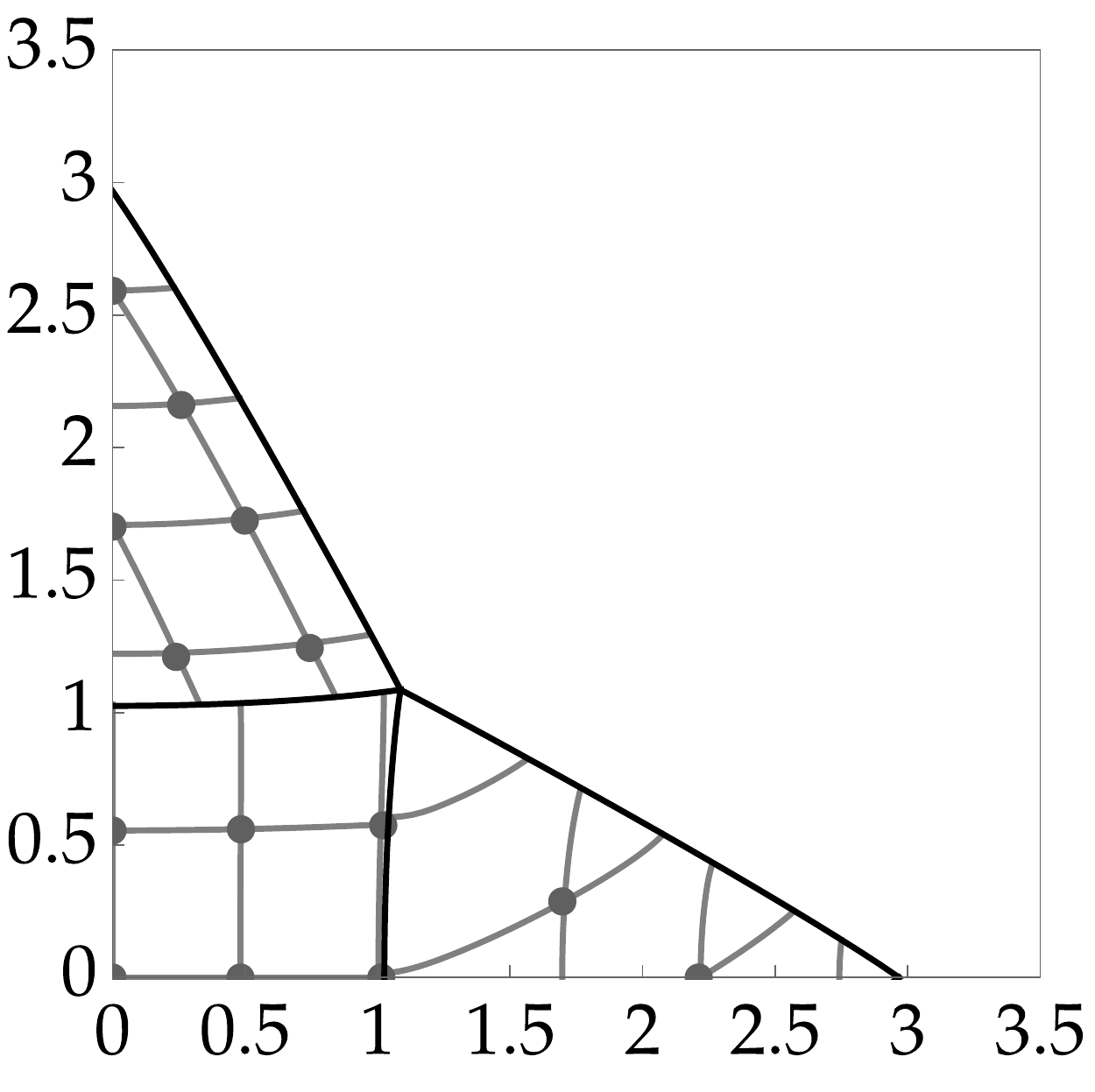}
\end{center}
\caption{$U=T^{-\frac{1}{2}}u^{[3]}_\mathrm{gO}(x;4,3)$; \\$\rho=\tfrac{3}{4}$; $\kappa=0$.}
\end{subfigure}%
\begin{subfigure}{1.5 in}
\begin{center}
\includegraphics[height=1.5 in]{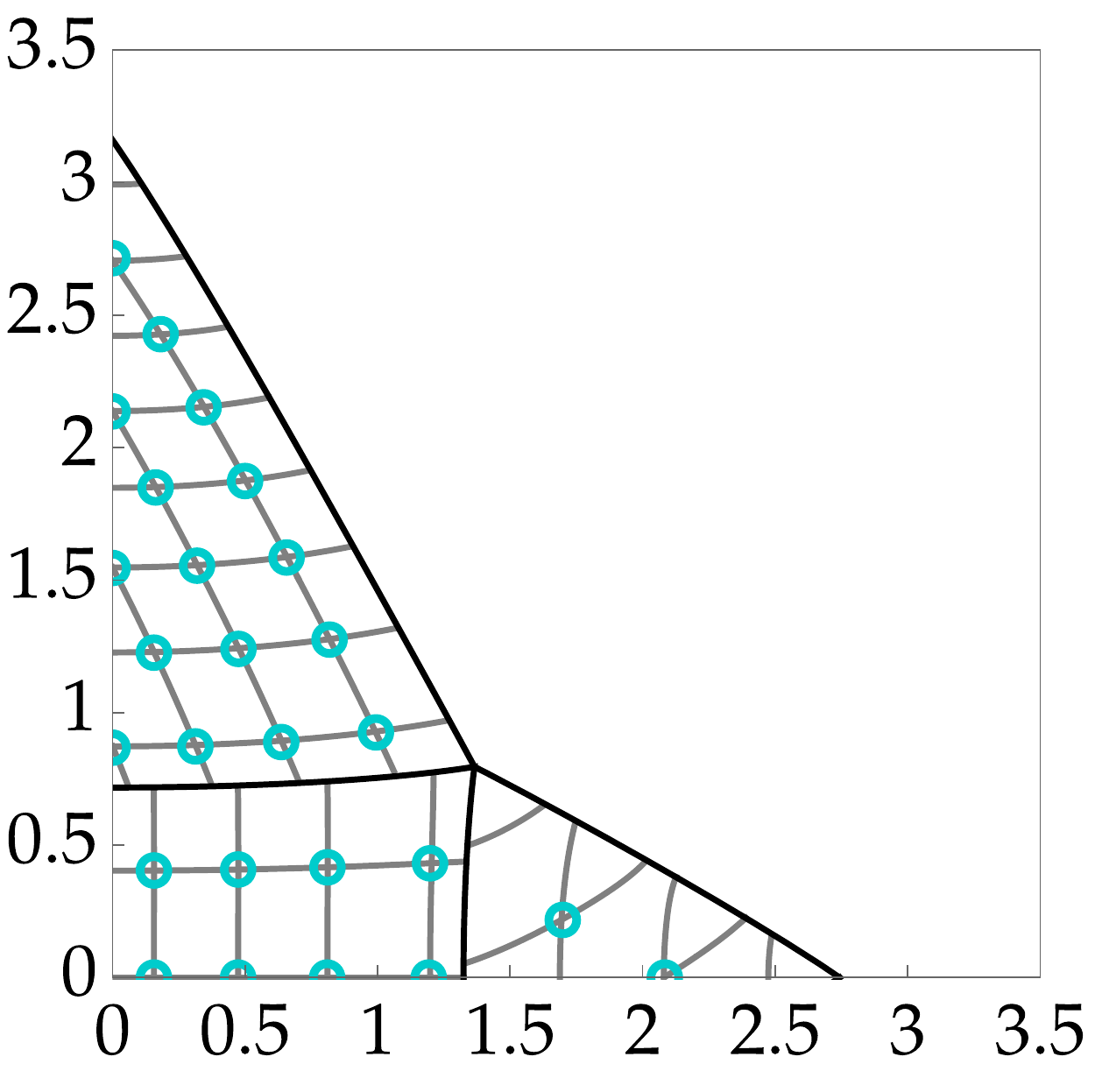}\\
\includegraphics[height=1.5 in]{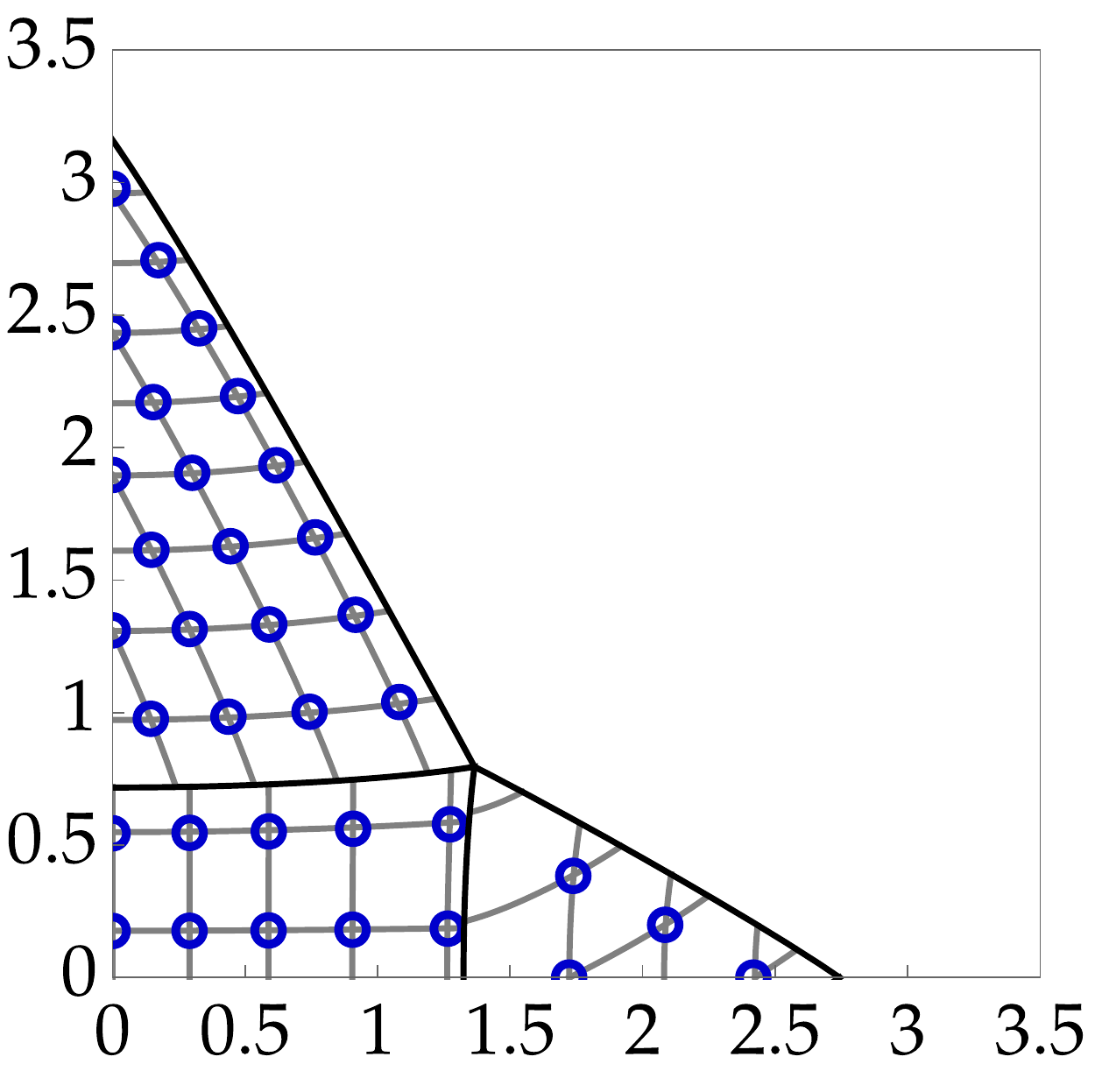}\\
\includegraphics[height=1.5 in]{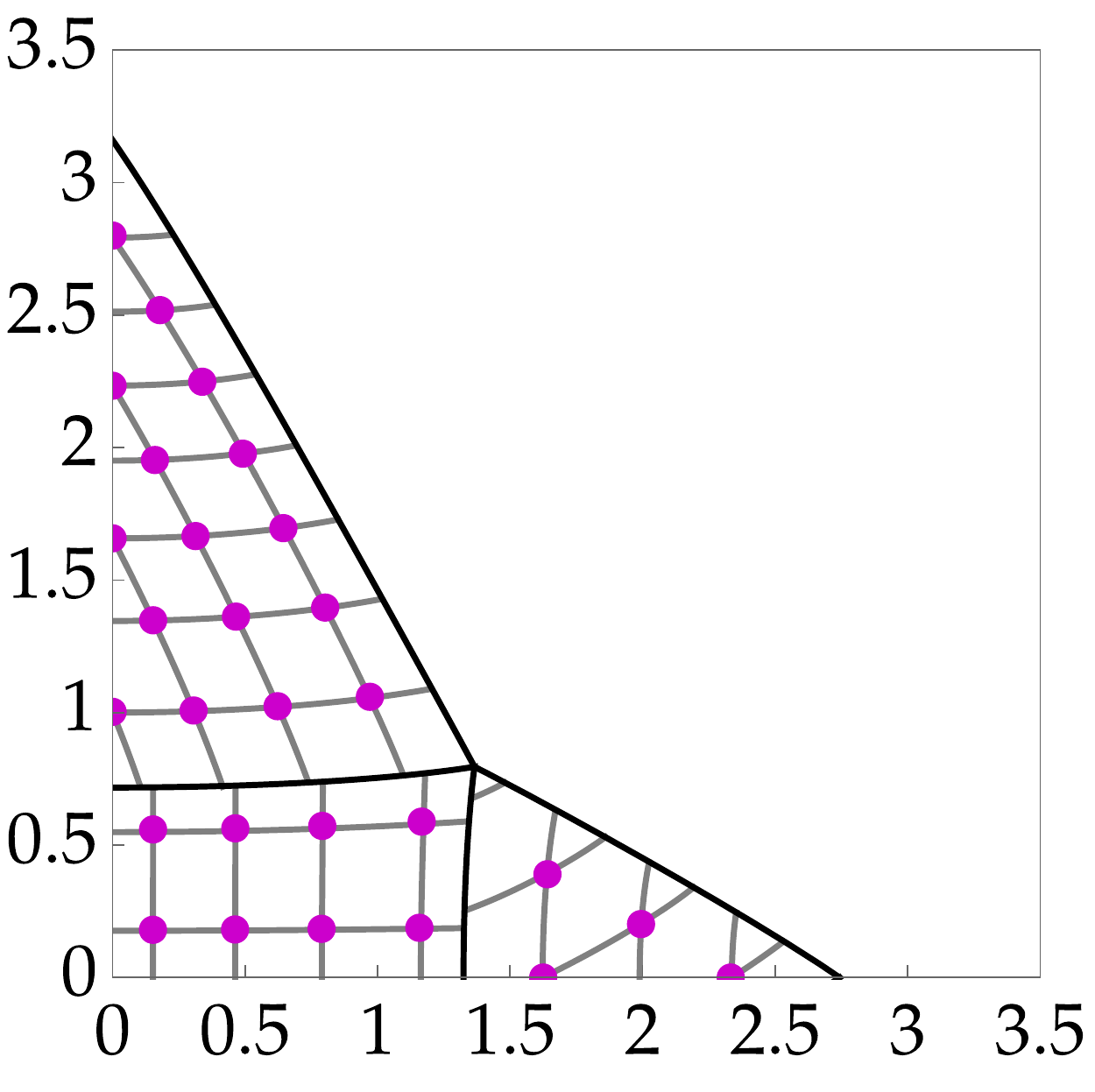}\\
\includegraphics[height=1.5 in]{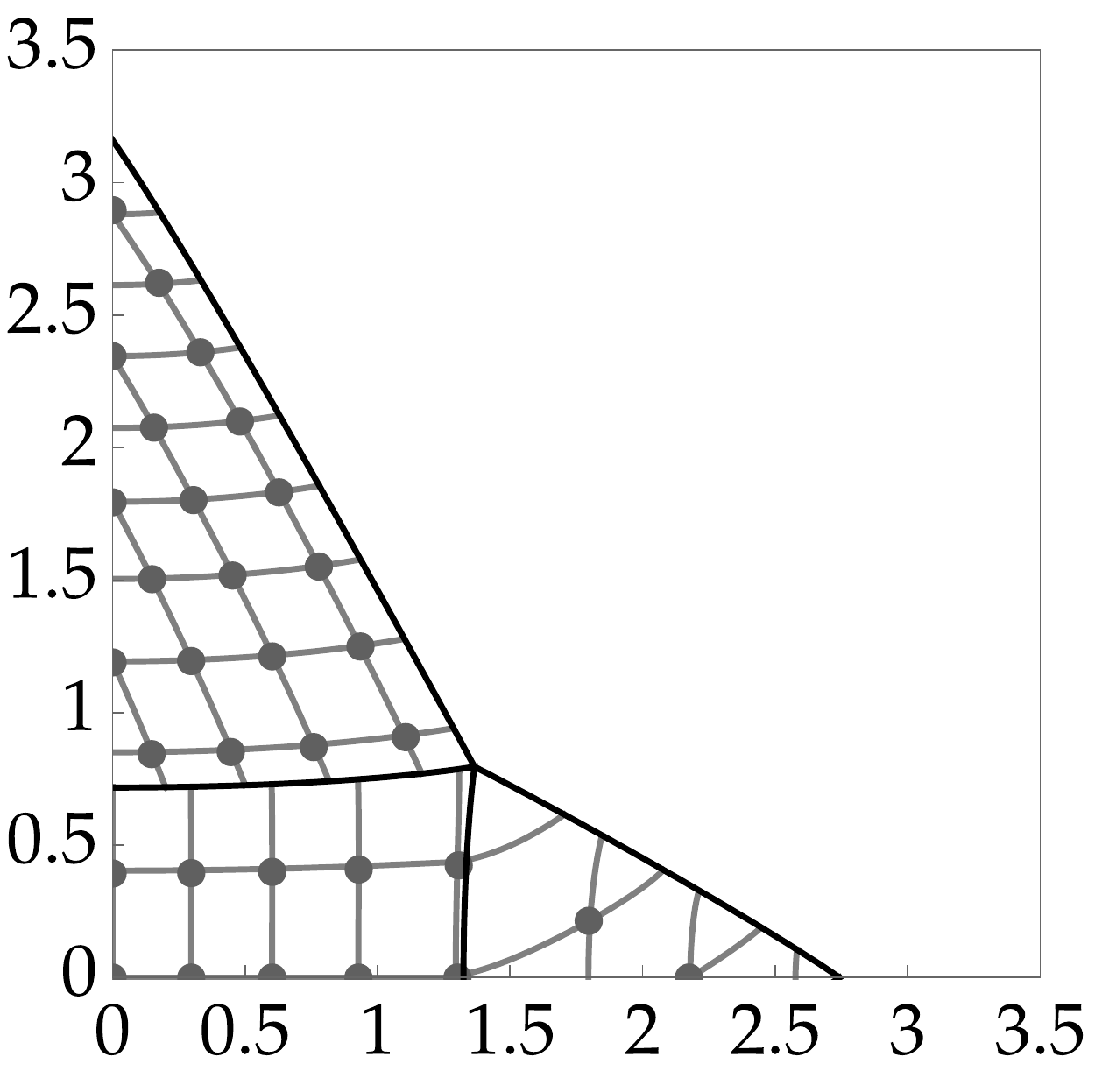}
\end{center}
\caption{$U=T^{-\frac{1}{2}}u^{[3]}_\mathrm{gO}(x;8,3)$; \\$\rho=\tfrac{3}{8}$; $\kappa=\tfrac{6}{17}$.}
\end{subfigure}
\end{center}
\caption{Quantitative comparison of zeros (circles, cyan for positive derivative and blue for negative derivative, top two rows of plots) and poles (dots, magenta for positive residue and gray for negative residue, bottom two rows of plots) of scaled gO rational solutions $U$ with the approximations given according to Corollary~\ref{cor:poles-and-zeros} by the intersection points of integer level curves (gray) of the two conditions in \eqref{eq:intro-zeros-quantize} or \eqref{eq:intro-poles-quantize} plotted in the $y=T^{-\frac{1}{2}}x$-plane.}
\label{fig:gOpp-poles-zeros}
\end{figure}
\begin{figure}[h]
\begin{center}
\begin{subfigure}{1.5 in}
\begin{center}
\includegraphics[height=1.5 in]{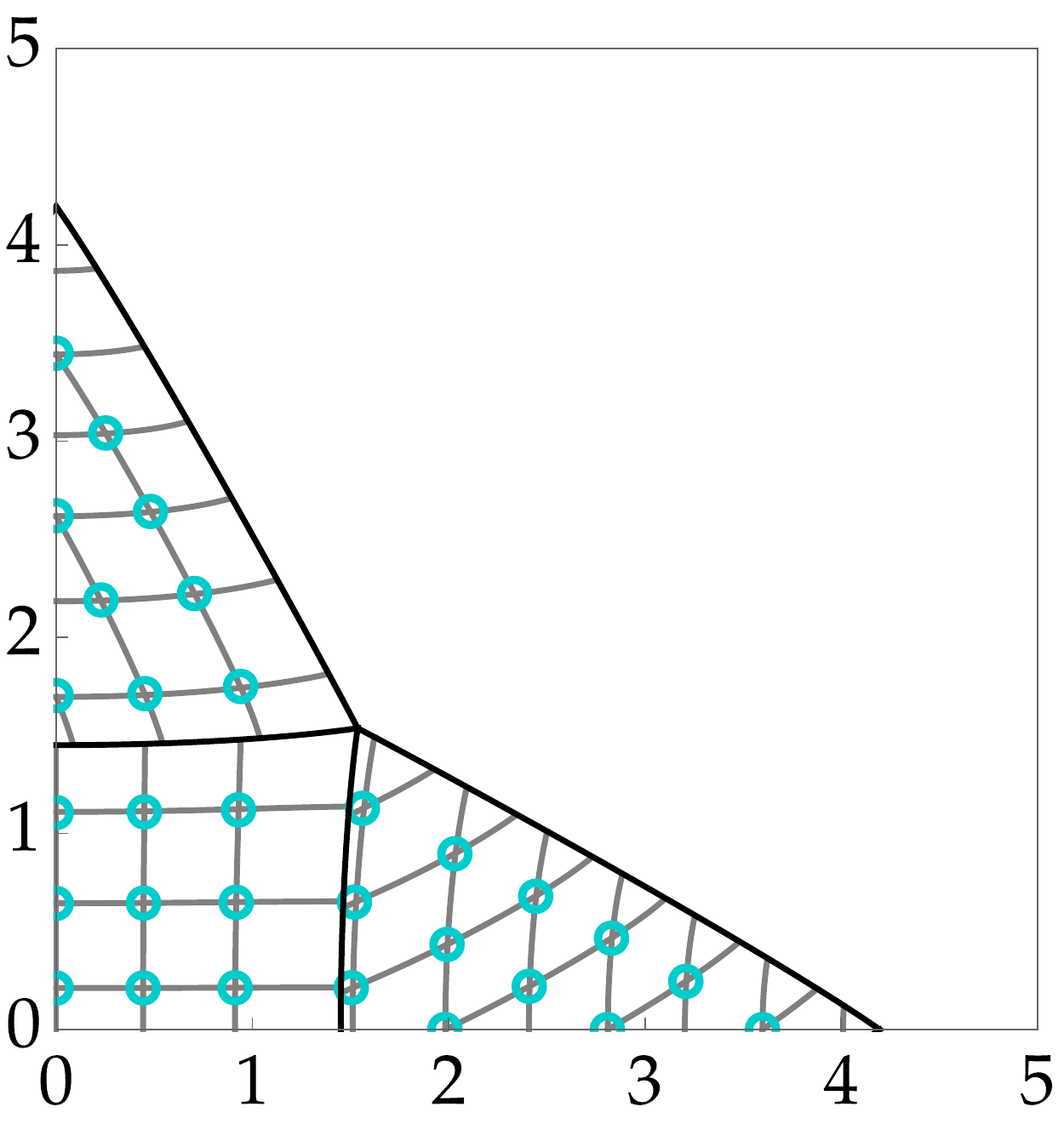}\\
\includegraphics[height=1.5 in]{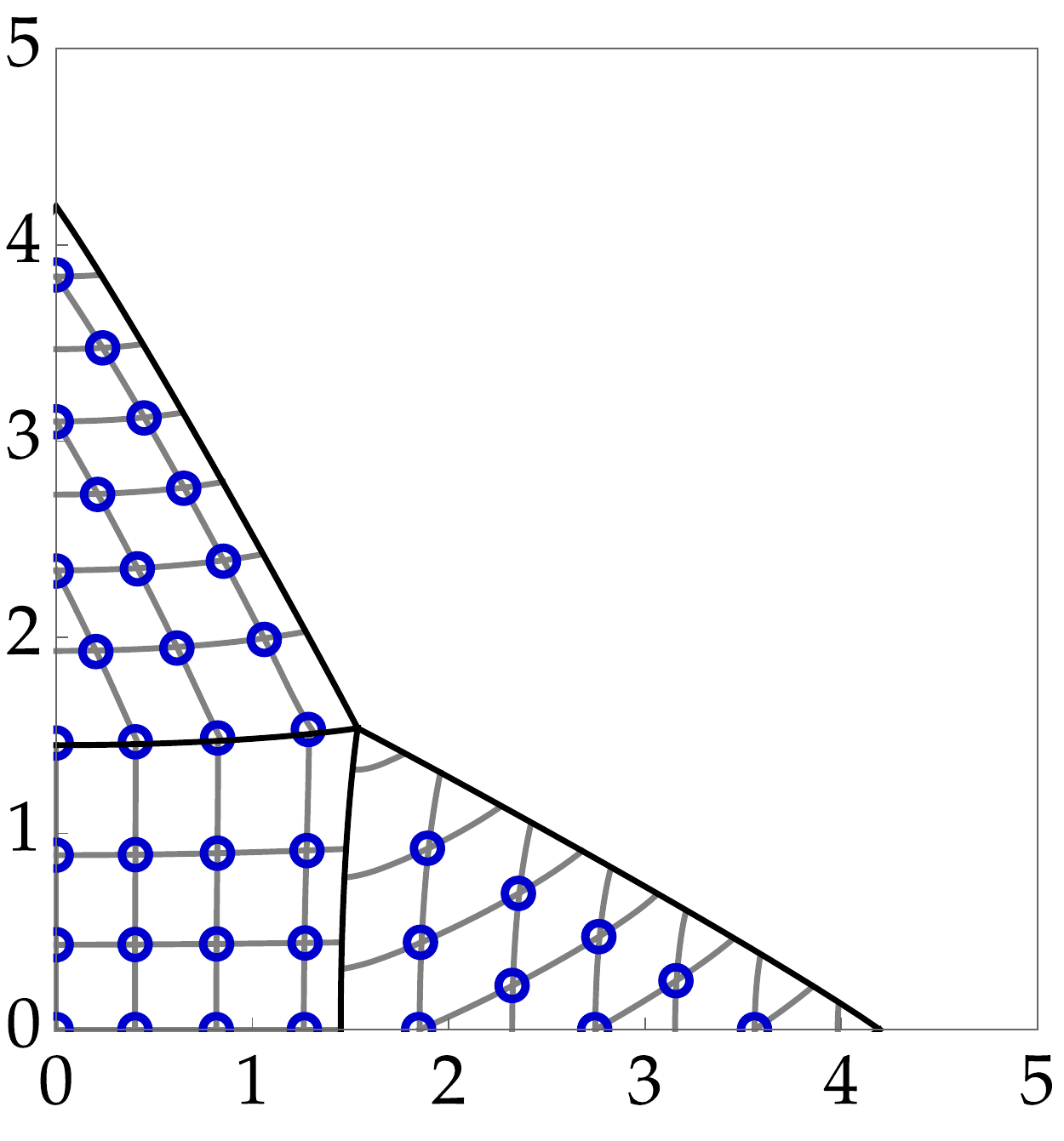}\\
\includegraphics[height=1.5 in]{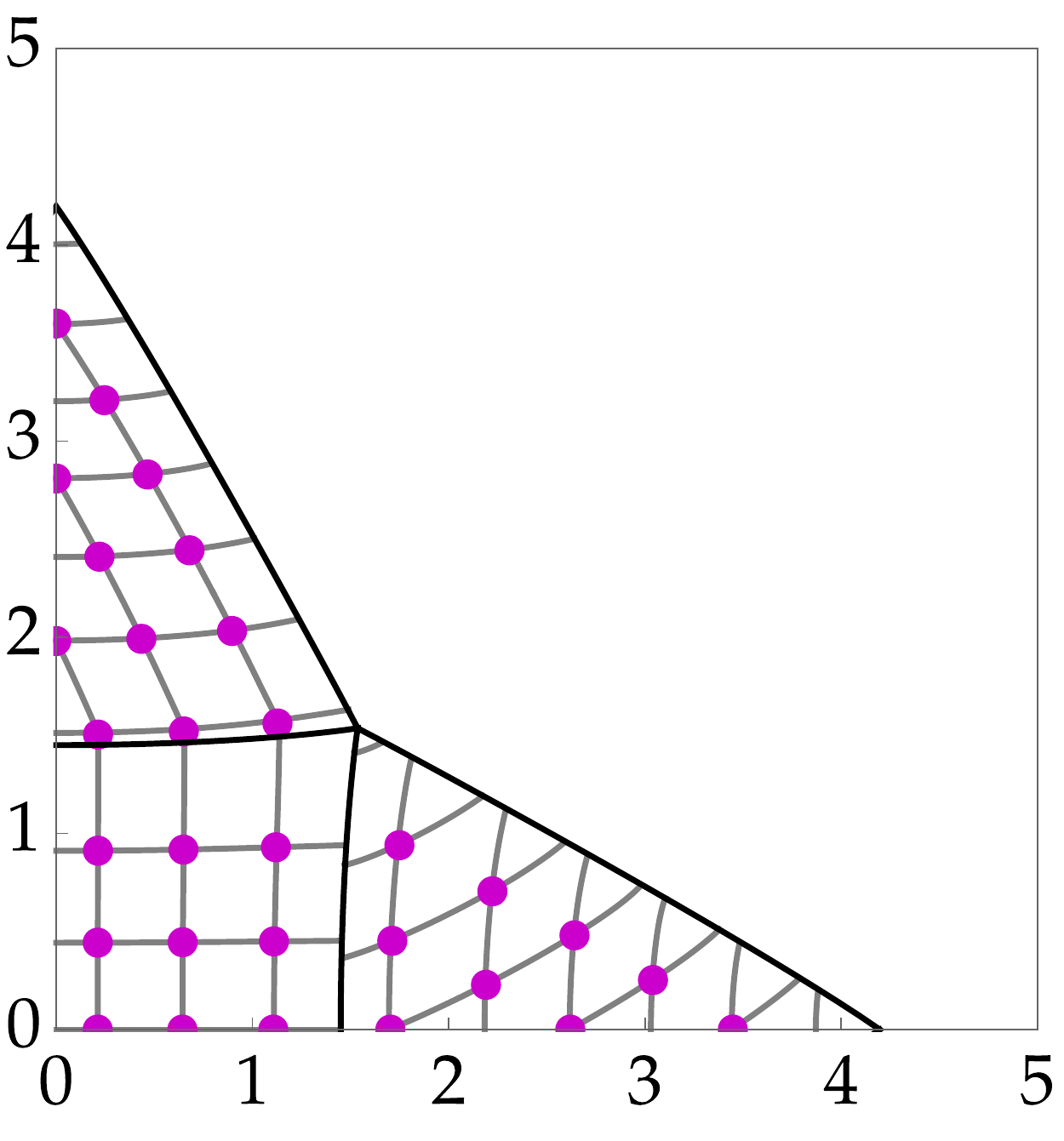}\\
\includegraphics[height=1.5 in]{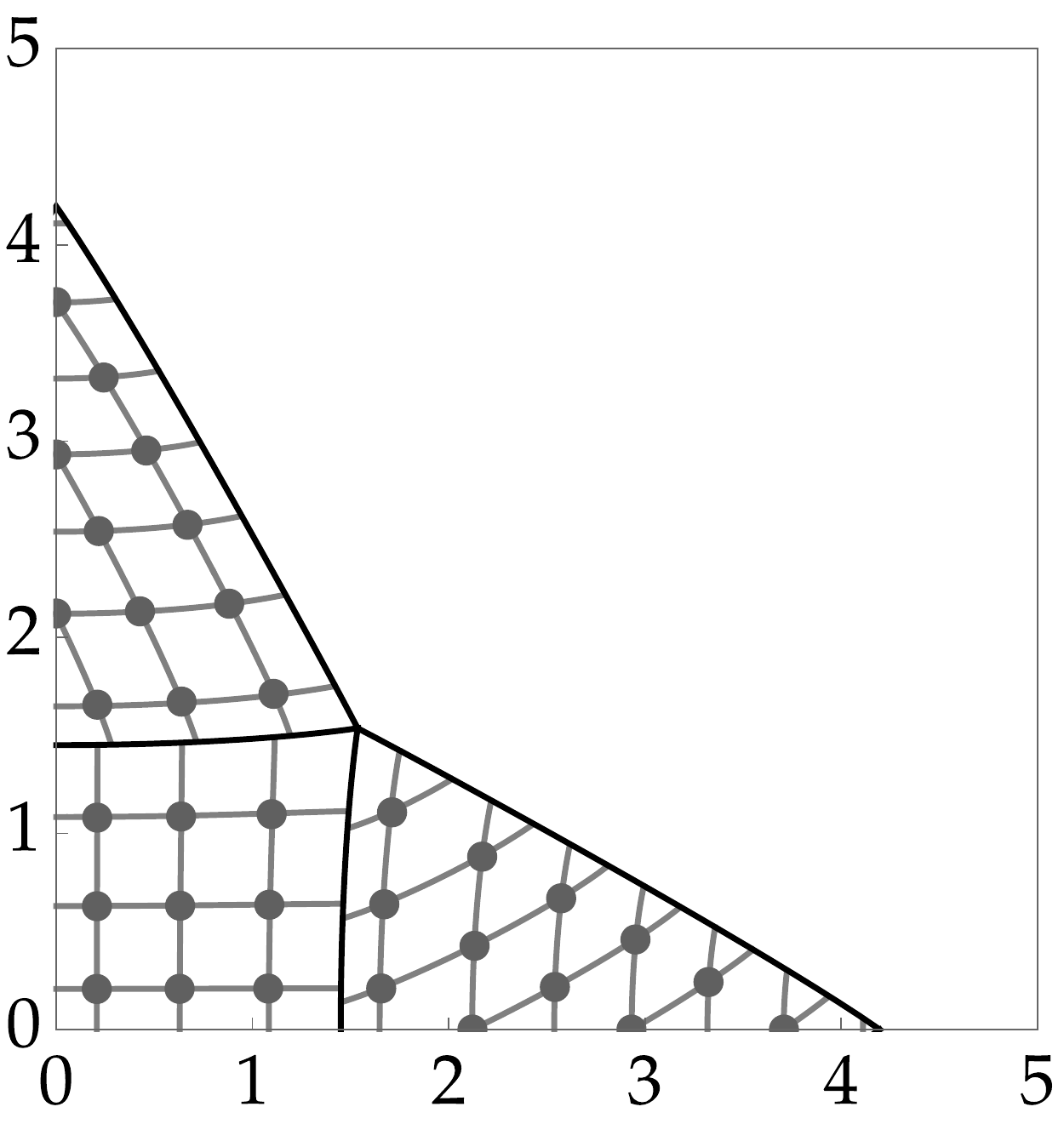}
\end{center}
\caption{$U=T^{-\frac{1}{2}}u^{[1]}_\mathrm{gO}(x;-6,-6)$; \\$\rho=1$; $\kappa=\tfrac{51}{19}$.}
\end{subfigure}%
\begin{subfigure}{1.5 in}
\begin{center}
\includegraphics[height=1.5 in]{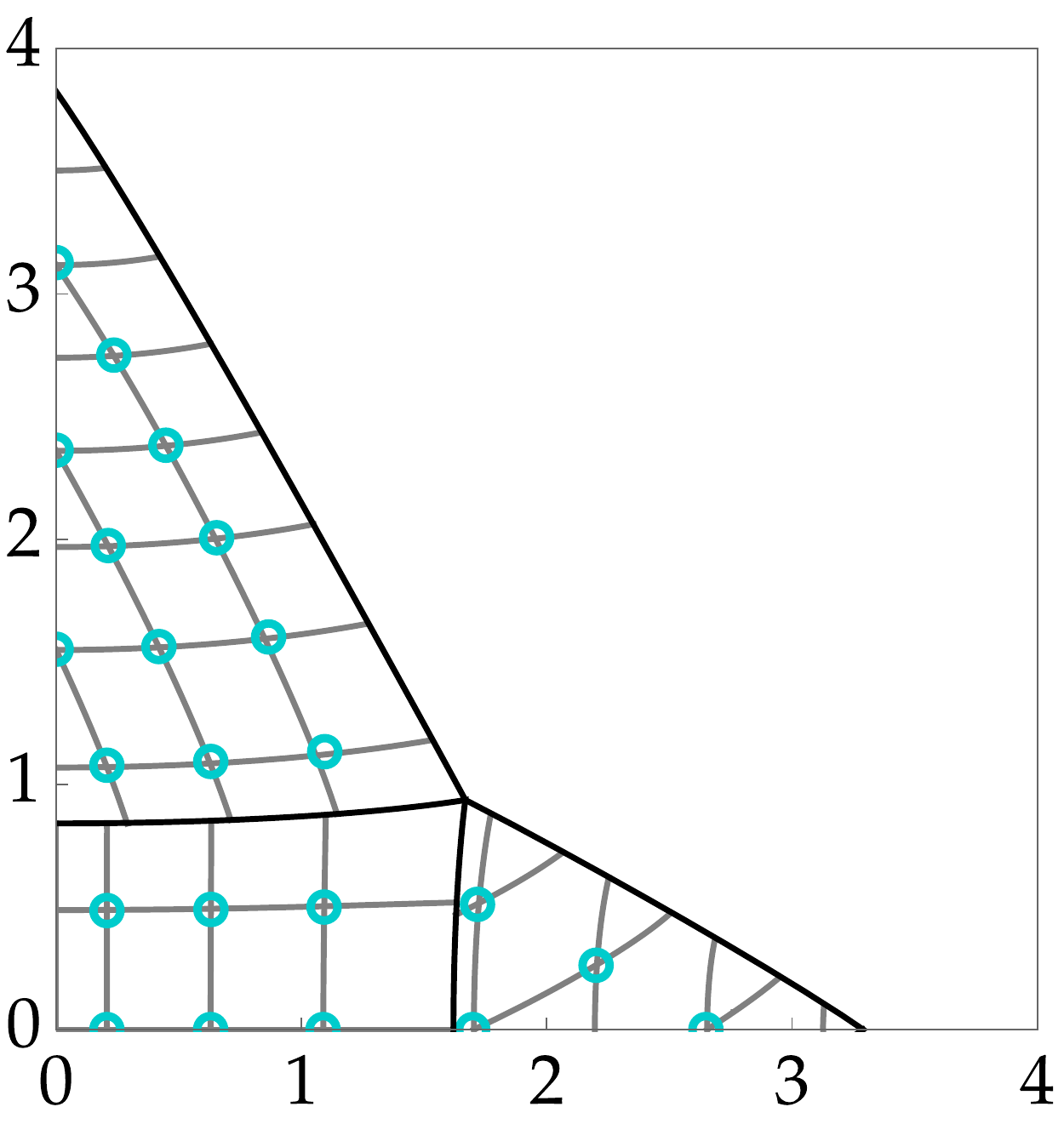}\\
\includegraphics[height=1.5 in]{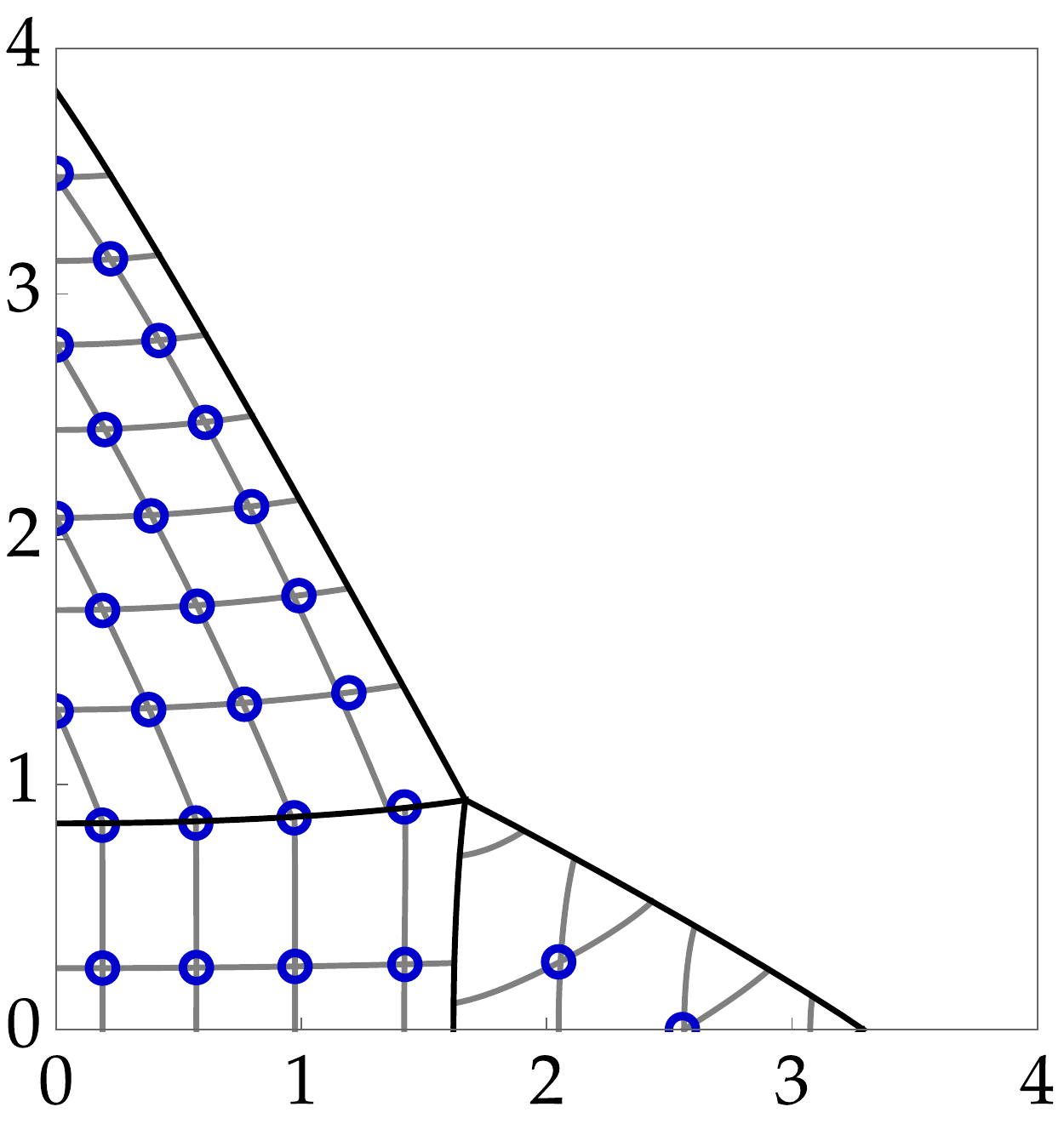}\\
\includegraphics[height=1.5 in]{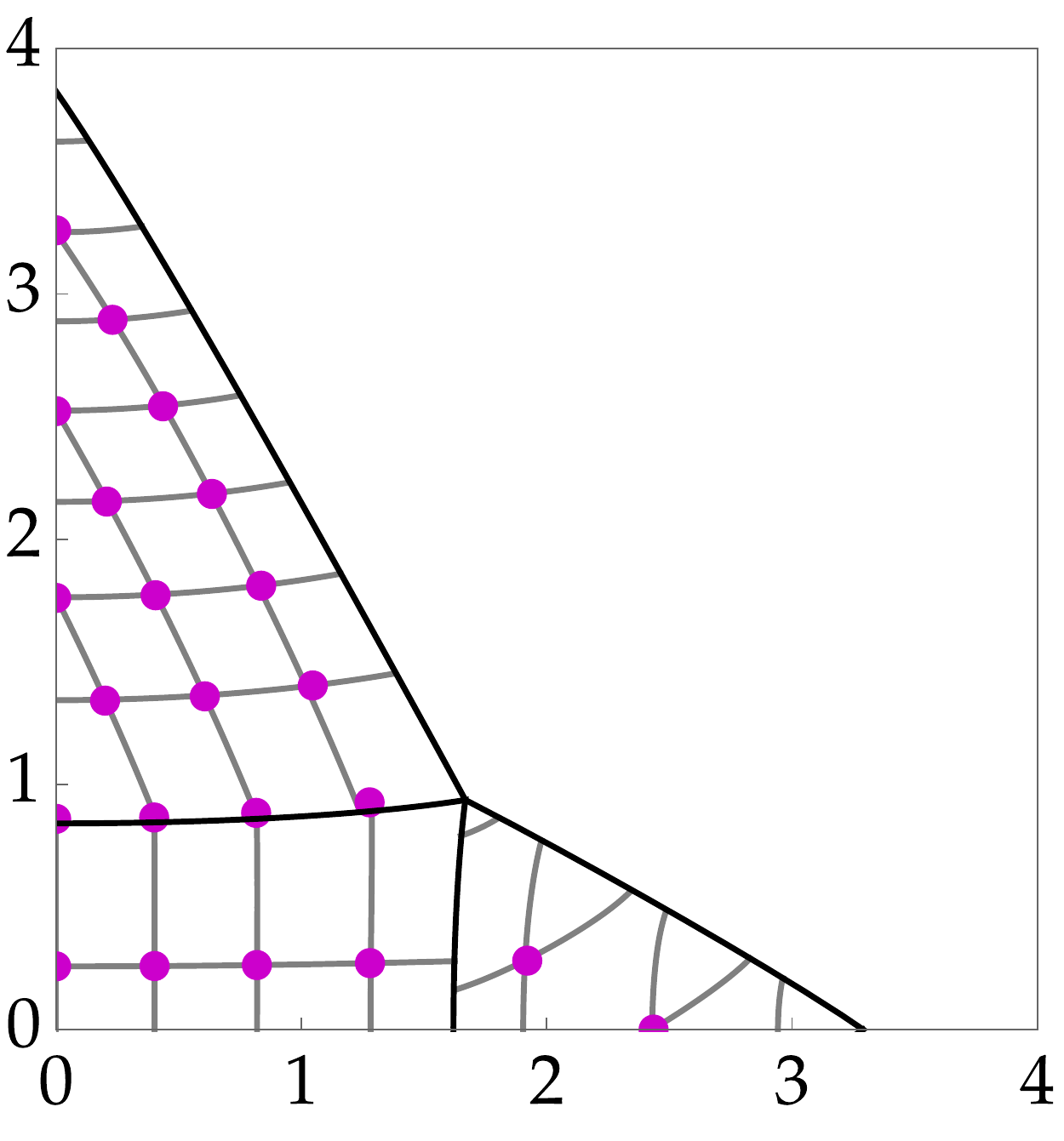}\\
\includegraphics[height=1.5 in]{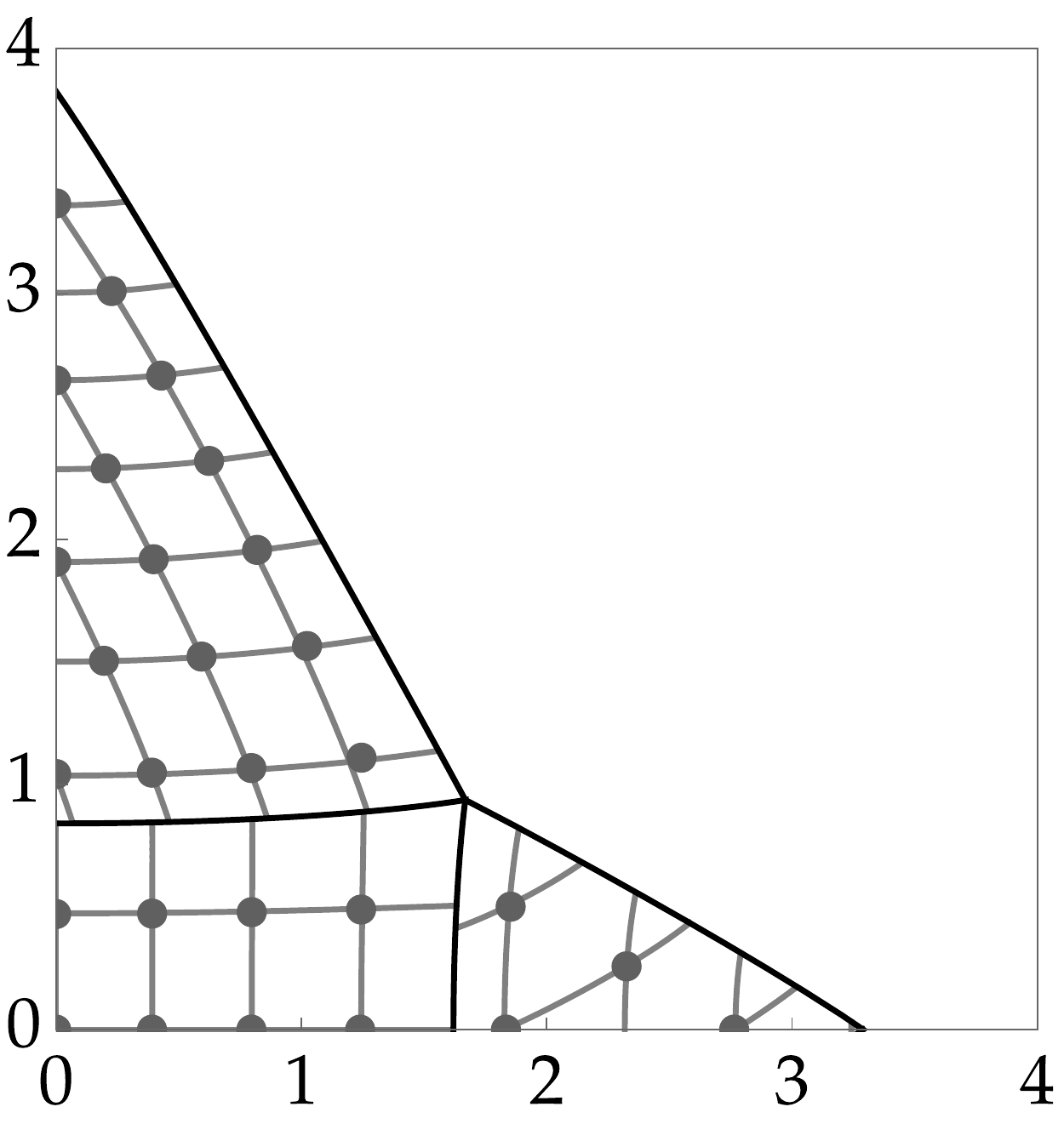}
\end{center}
\caption{$U=T^{-\frac{1}{2}}u^{[1]}_\mathrm{gO}(x;-3,-7)$; \\$\rho=\tfrac{7}{3}$; $\kappa=\tfrac{18}{11}$.}
\end{subfigure}%
\begin{subfigure}{1.5 in}
\begin{center}
\includegraphics[height=1.5 in]{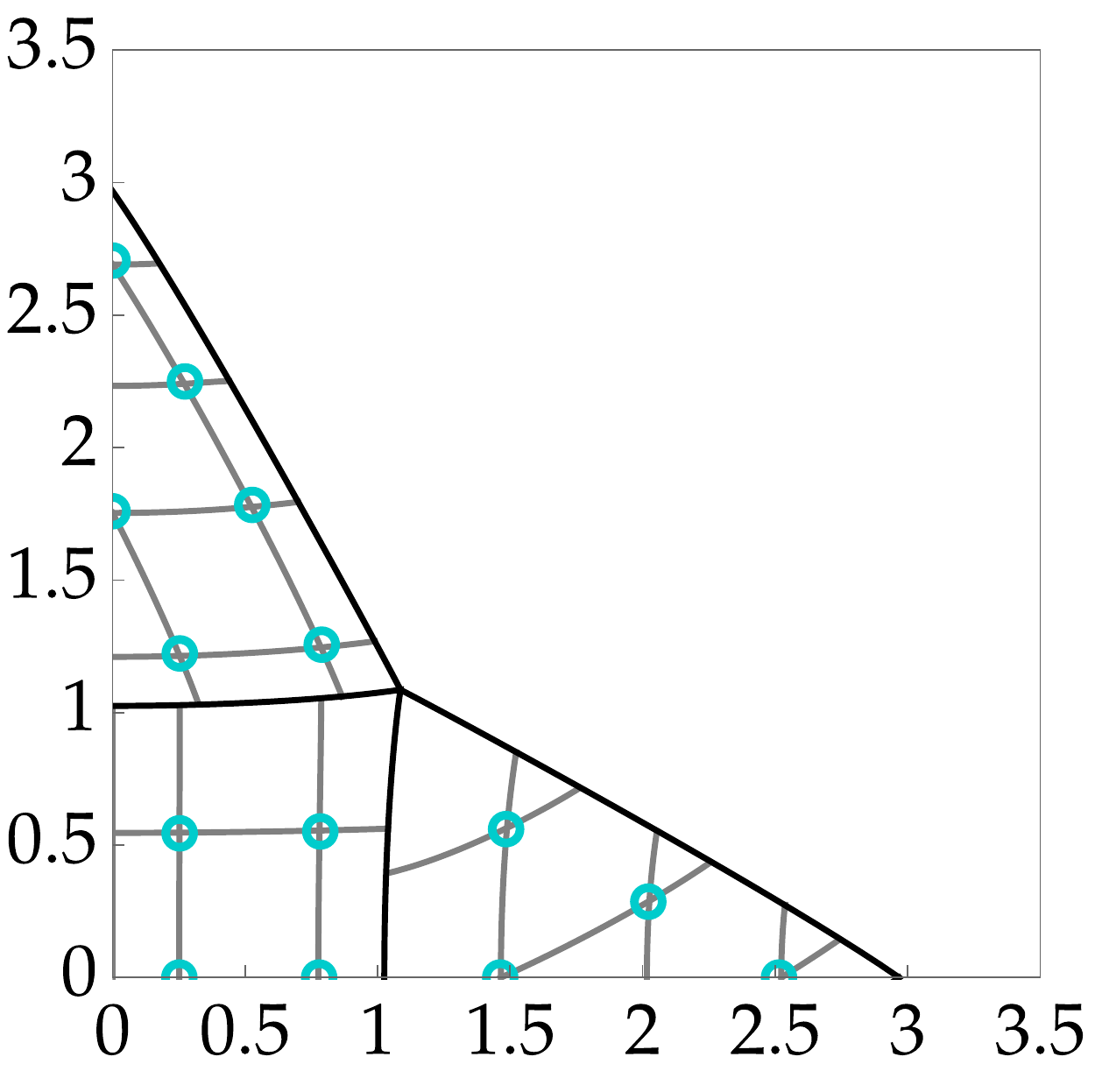}\\
\includegraphics[height=1.5 in]{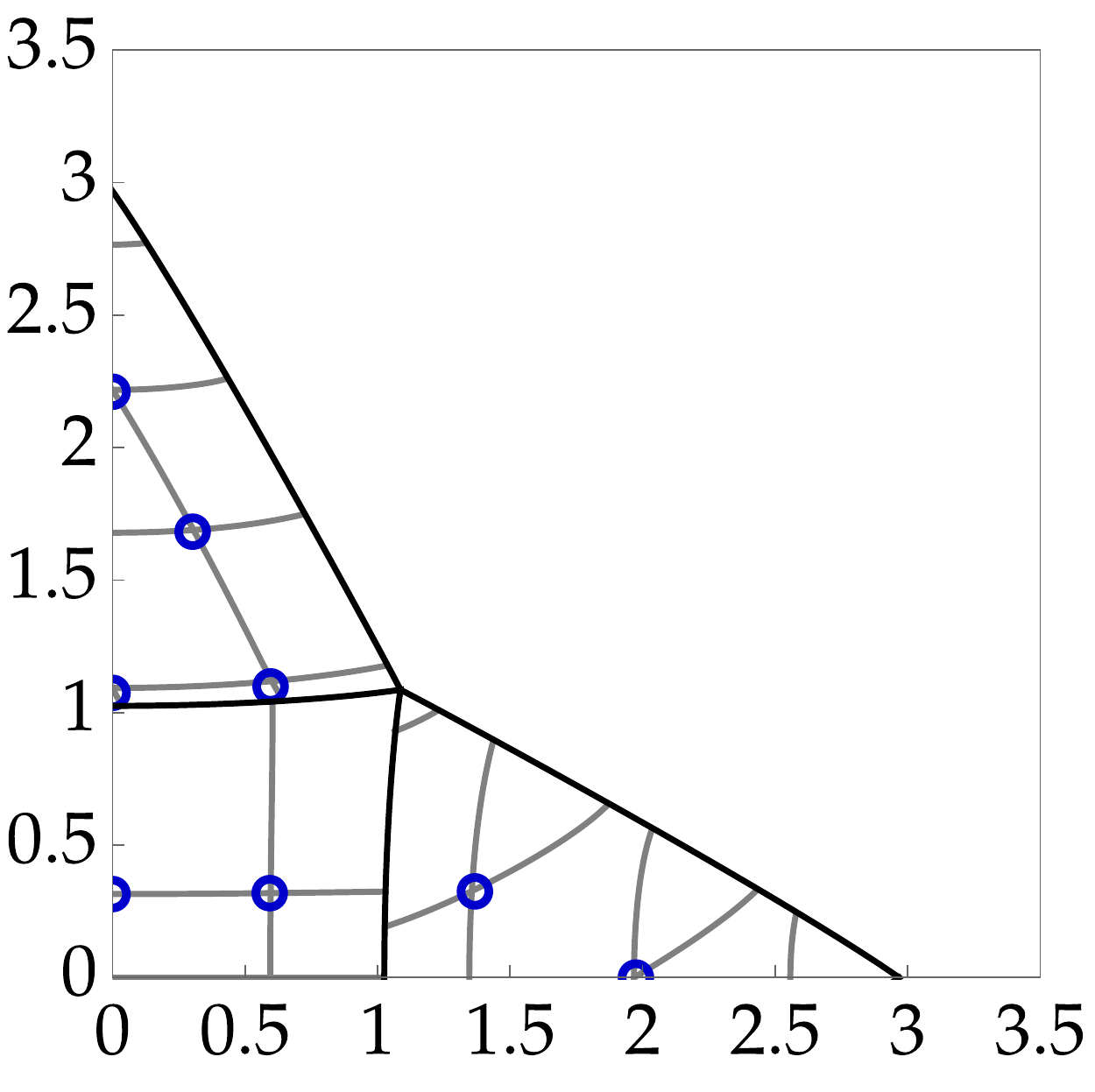}\\
\includegraphics[height=1.5 in]{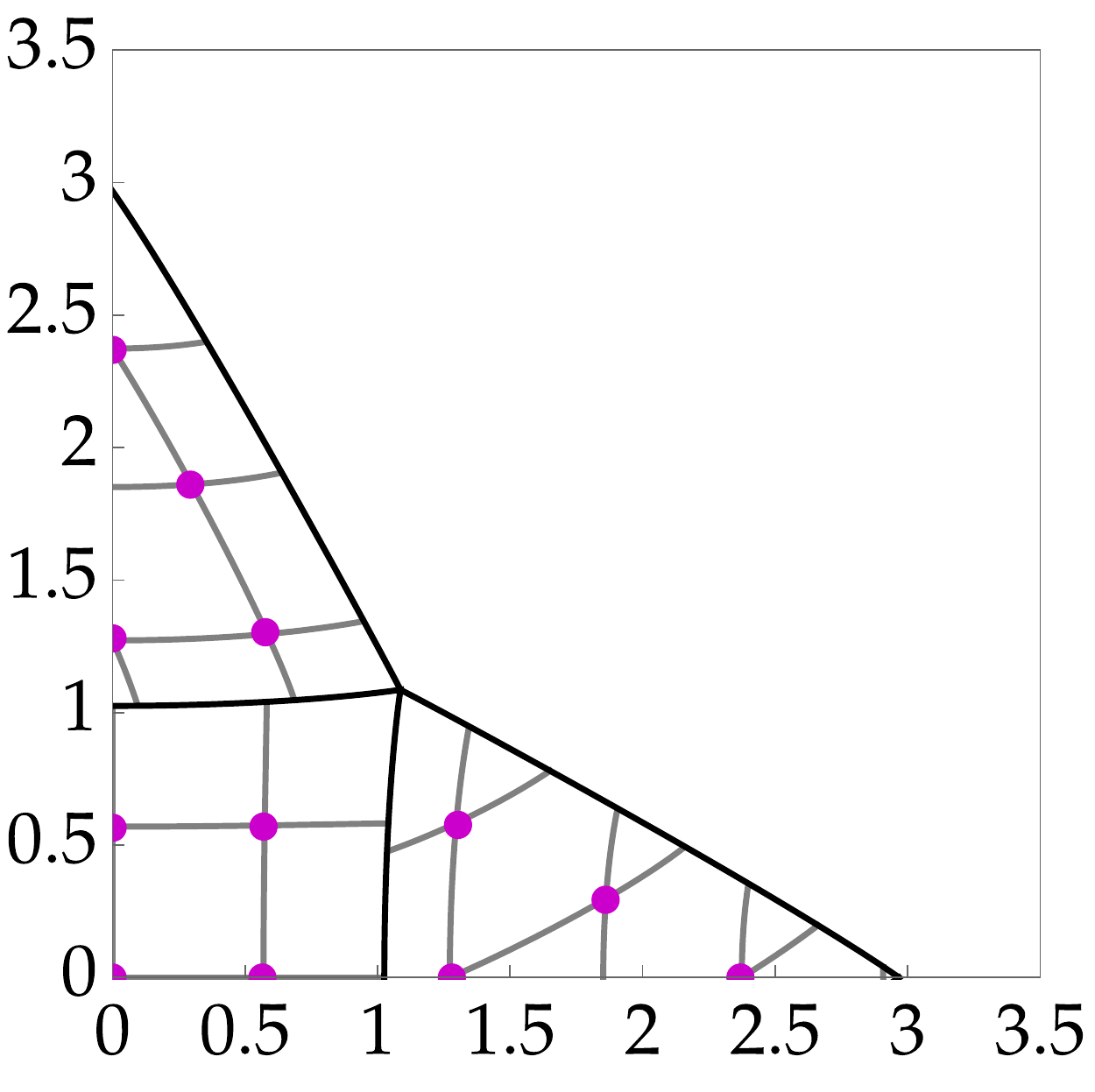}\\
\includegraphics[height=1.5 in]{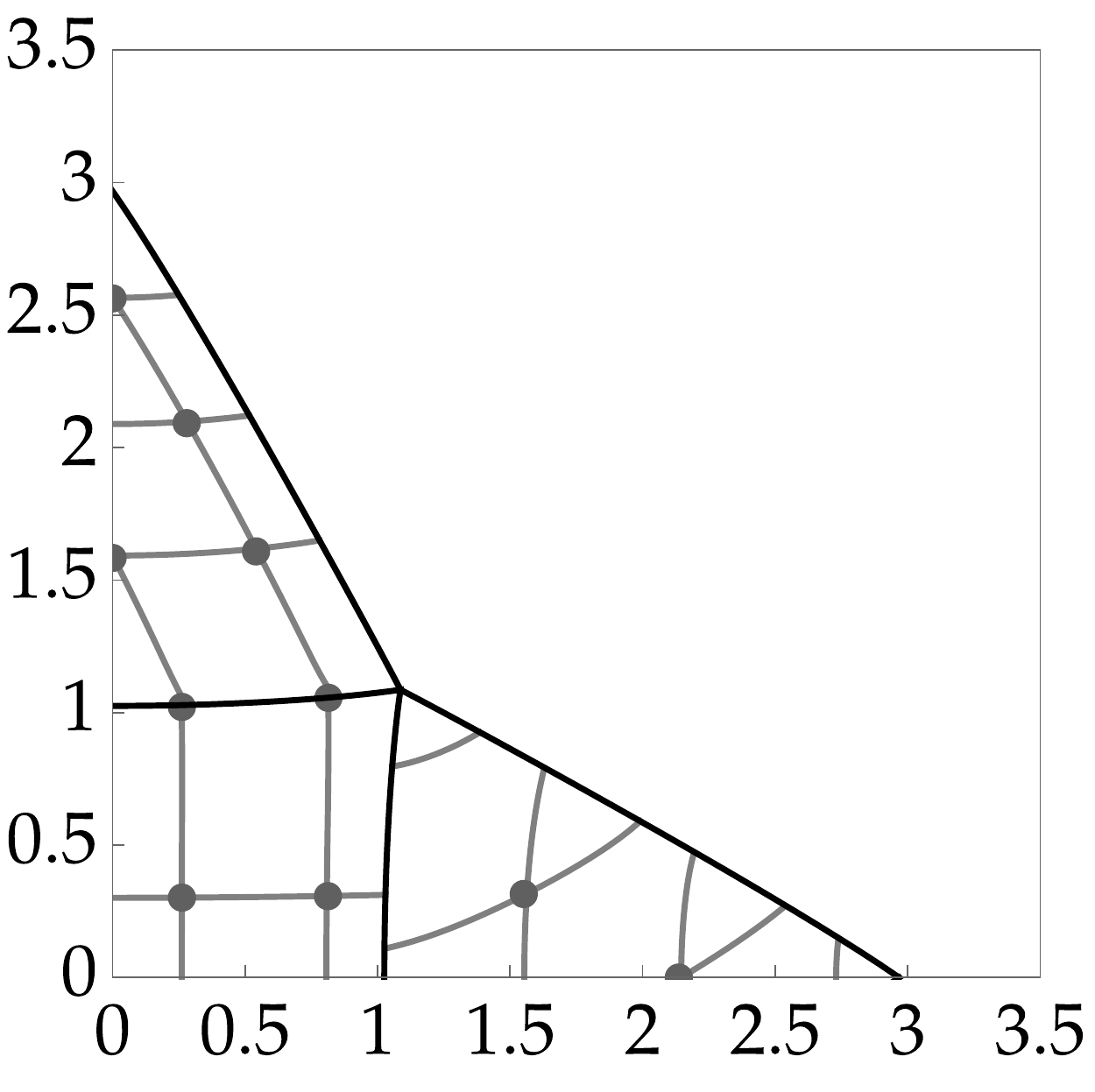}
\end{center}
\caption{$U=T^{-\frac{1}{2}}u^{[3]}_\mathrm{gO}(x;-3,-4)$; \\$\rho=\tfrac{4}{3}$; $\kappa=0$.}
\end{subfigure}
\begin{subfigure}{1.5 in}
\begin{center}
\includegraphics[height=1.5 in]{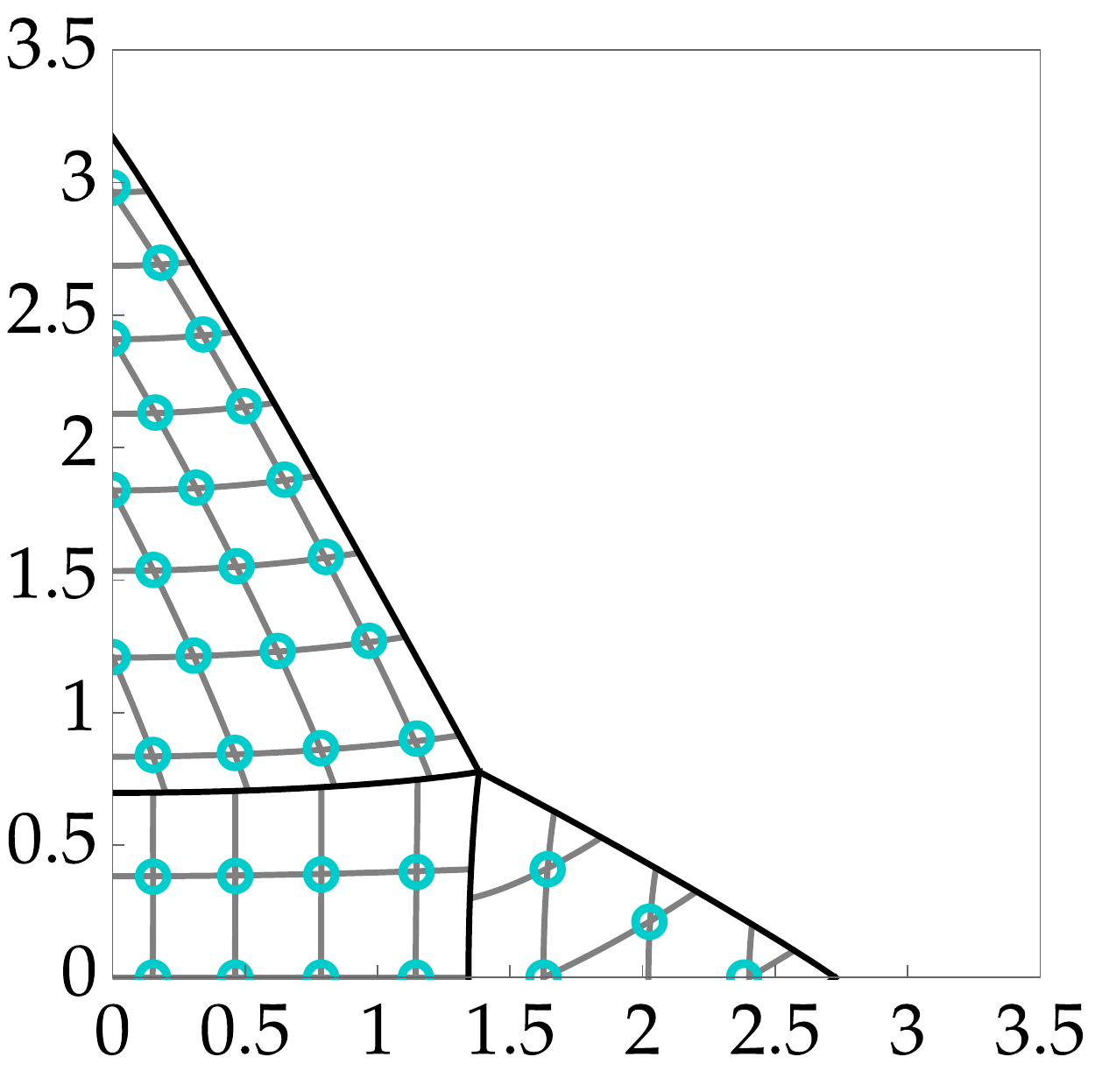}\\
\includegraphics[height=1.5 in]{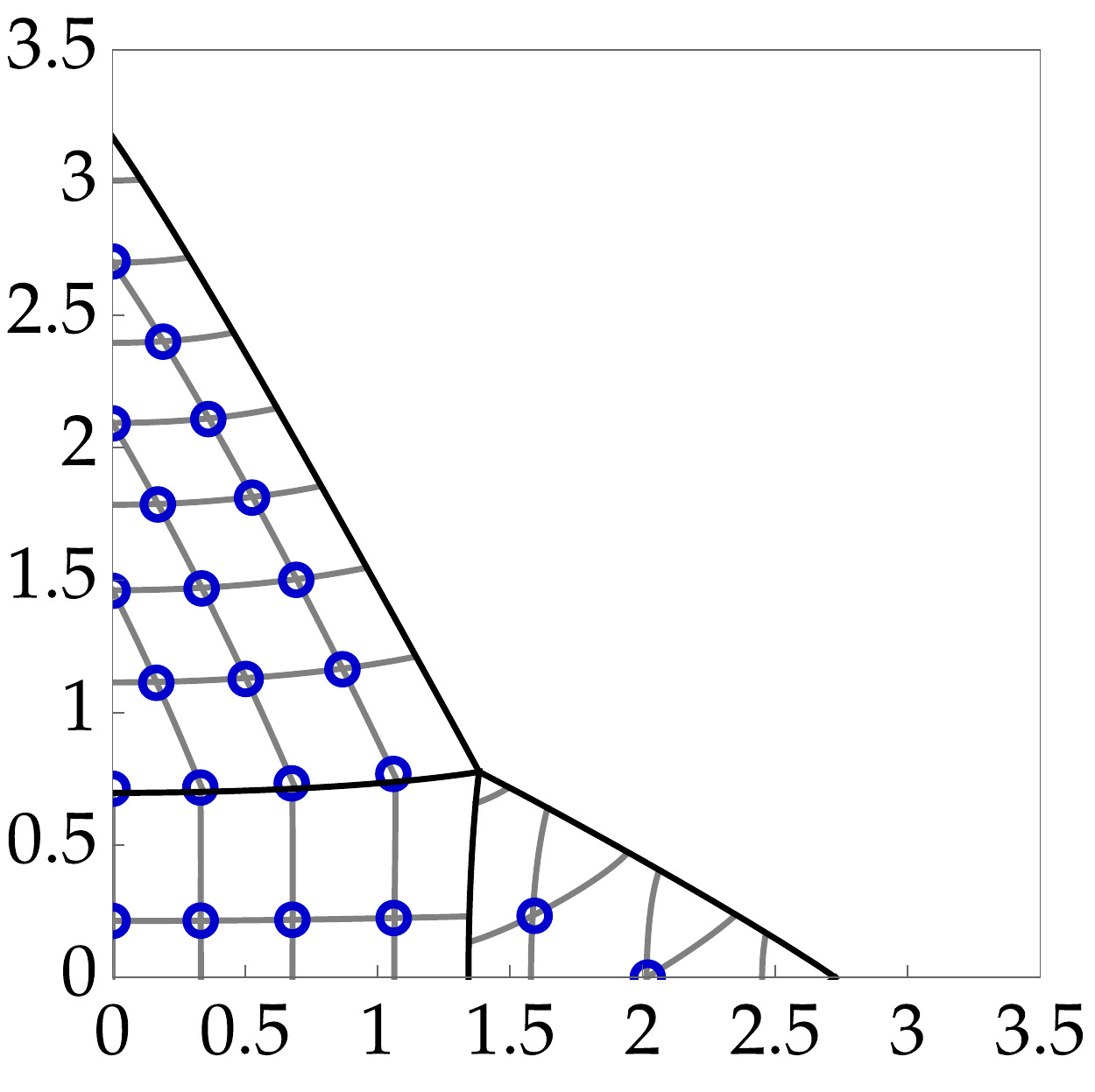}\\
\includegraphics[height=1.5 in]{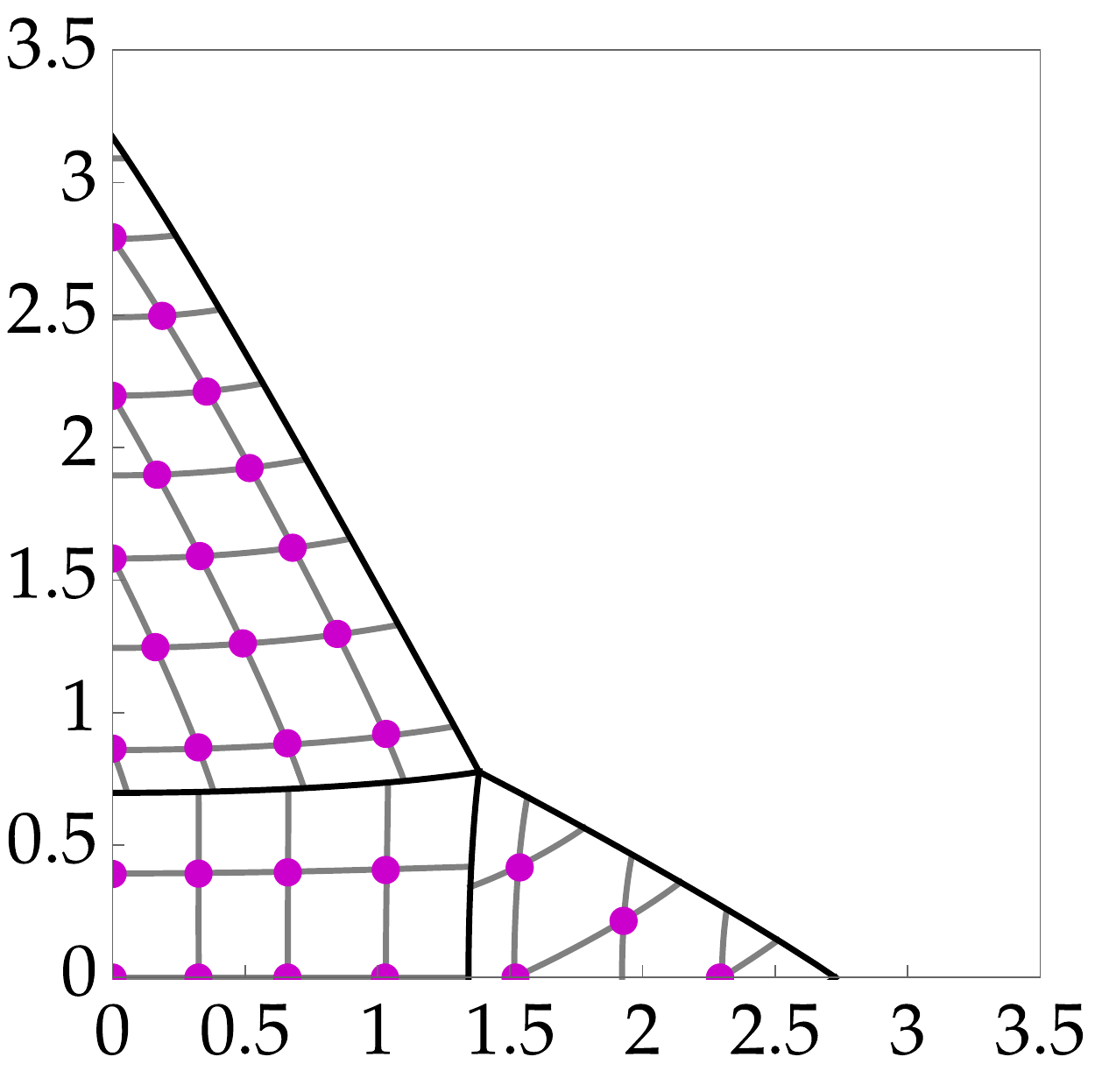}\\
\includegraphics[height=1.5 in]{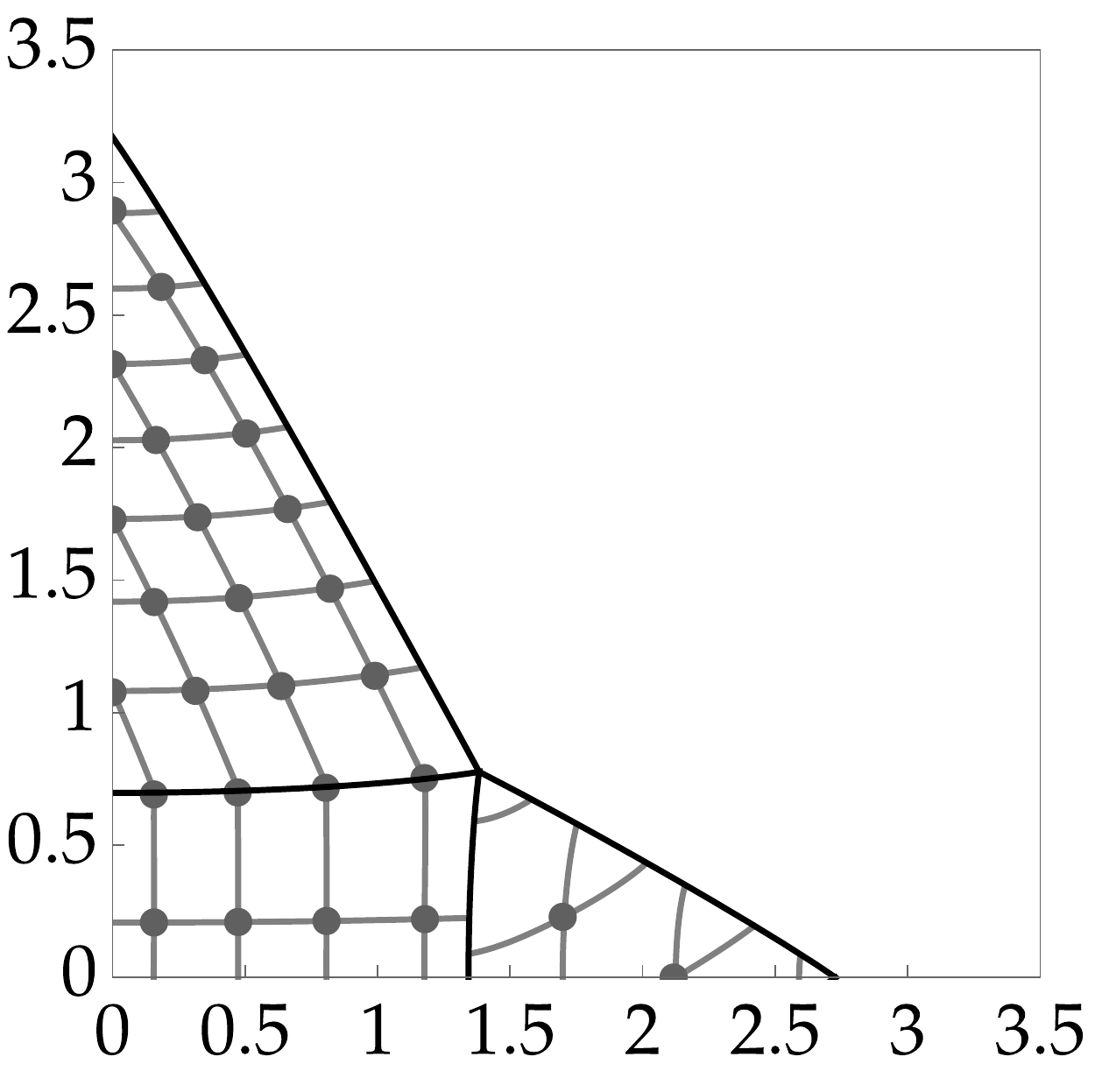}
\end{center}
\caption{$U=T^{-\frac{1}{2}}u^{[3]}_\mathrm{gO}(x;-3,-8)$; \\$\rho=\tfrac{8}{3}$; $\kappa=\tfrac{3}{8}$.}
\end{subfigure}%
\end{center}
\caption{As in Figure~\ref{fig:gOpp-axes} but for negative indices $(m,n)$.}
\label{fig:gOnn-poles-zeros}
\end{figure}

\begin{remark}
In the special case that $y_0=0\in\rectangle(\kappa)$, from Proposition~\ref{prop:E} we have that also $E=0$.  Furthermore, taking into account the identities \eqref{eq:R2-rectangle-y0-real} and \eqref{eq:R1-rectangle-y0-imaginary}, both of which hold if and only if $y_0=0$, as well as the condition that $(\Theta_0,\Theta_\infty)\in\Lambda_\mathrm{gH}$ or $(\Theta_0,\Theta_\infty)\in\Lambda_\mathrm{gO}$, one can show that the prediction of Corollary~\ref{cor:poles-and-zeros} is exact in this special case.  In other words, the pole or zero of $u_\mathrm{F}^{[j]}(x;m,n)$ that must lie at the origin according to Proposition~\ref{prop:FirstQuadrant} is captured exactly by the approximation formul\ae\ of Theorems~\ref{thm:Hermite-elliptic}--\ref{thm:Okamoto-elliptic}.
\label{rem:exact-at-origin}
\end{remark}
If we use the approach (ii) in which a rational solution of Painlev\'e-IV is approximated locally near a given point $y_0$ by an exact elliptic function of $\zeta$ with a uniform lattice of poles and zeros, then we can illustrate the attraction of the actual poles and zeros to the uniform lattice.  The accuracy of this type of approximation of poles and zeros is shown for the gH family with $y_0=0$ in Figures~\ref{fig:gH-poles-zeros-type-1}--\ref{fig:gH-poles-zeros-type-3}.
\begin{figure}
\begin{center}
\begin{subfigure}{1.5 in}
\includegraphics[width=1.5in]{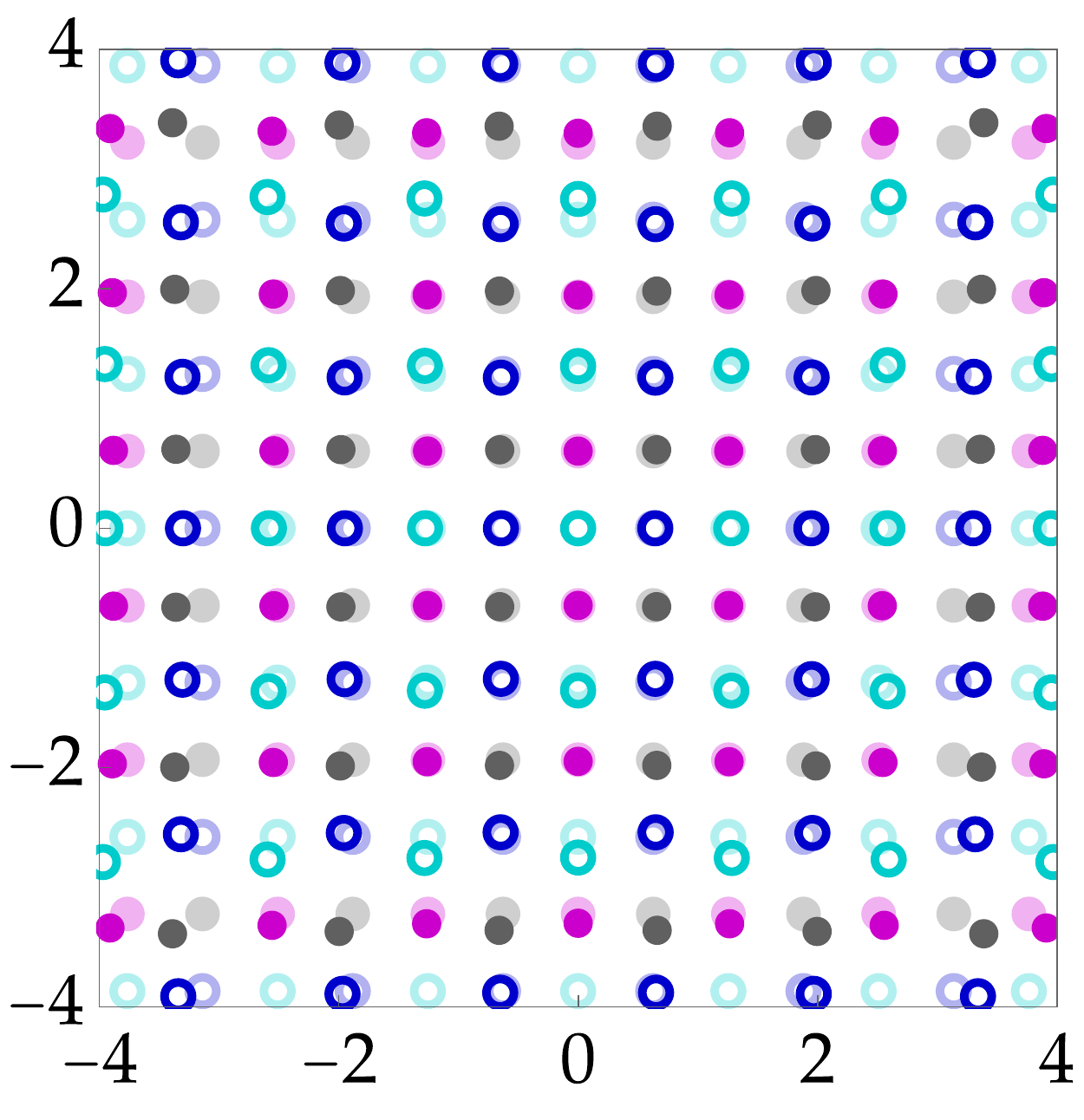}
\caption{$(m,n)=(8,8)$}
\end{subfigure}%
\begin{subfigure}{1.5 in}
\includegraphics[width=1.5in]{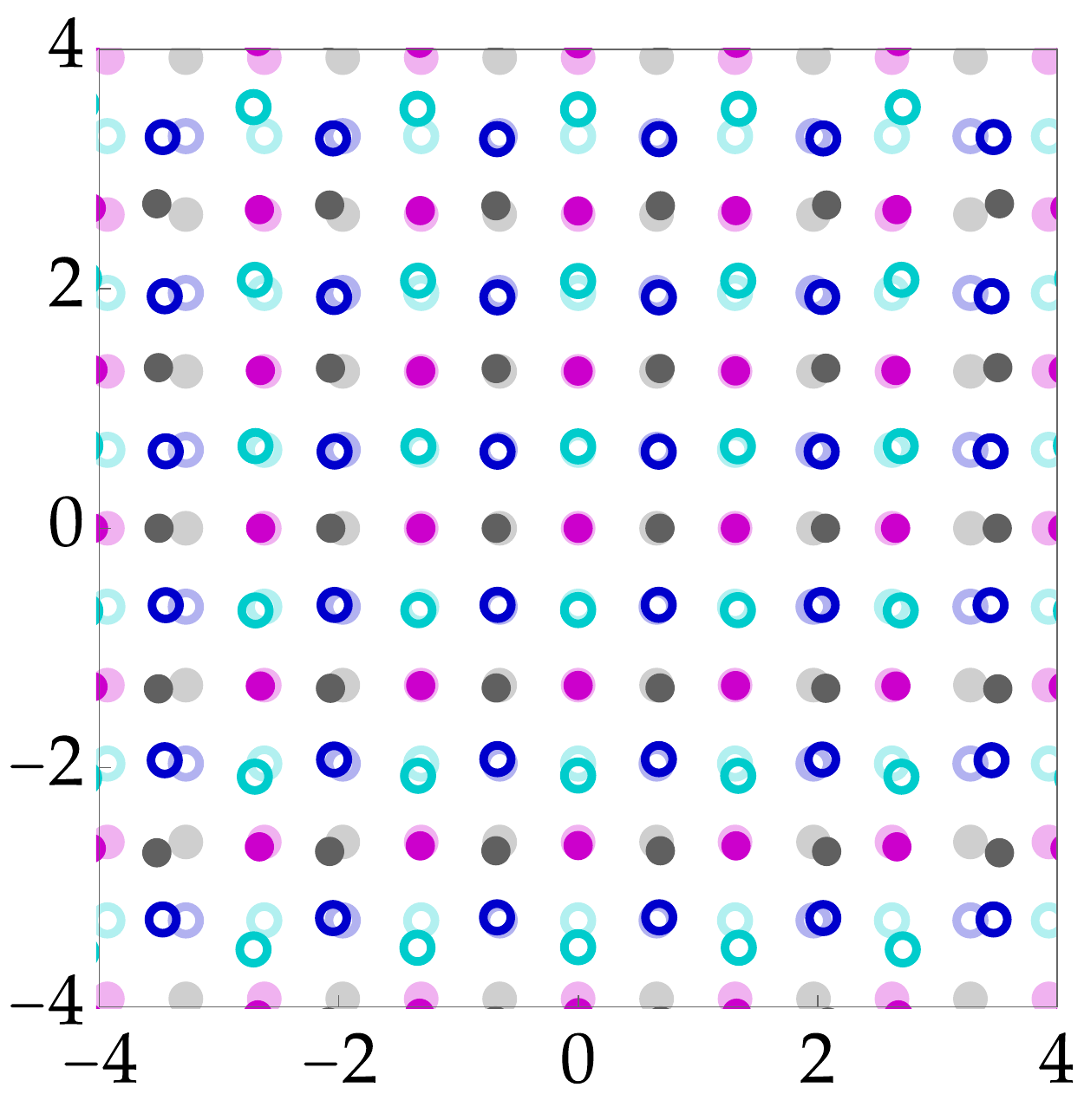}
\caption{$(m,n)=(8,9)$}
\end{subfigure}%
\begin{subfigure}{1.5 in}
\includegraphics[width=1.5in]{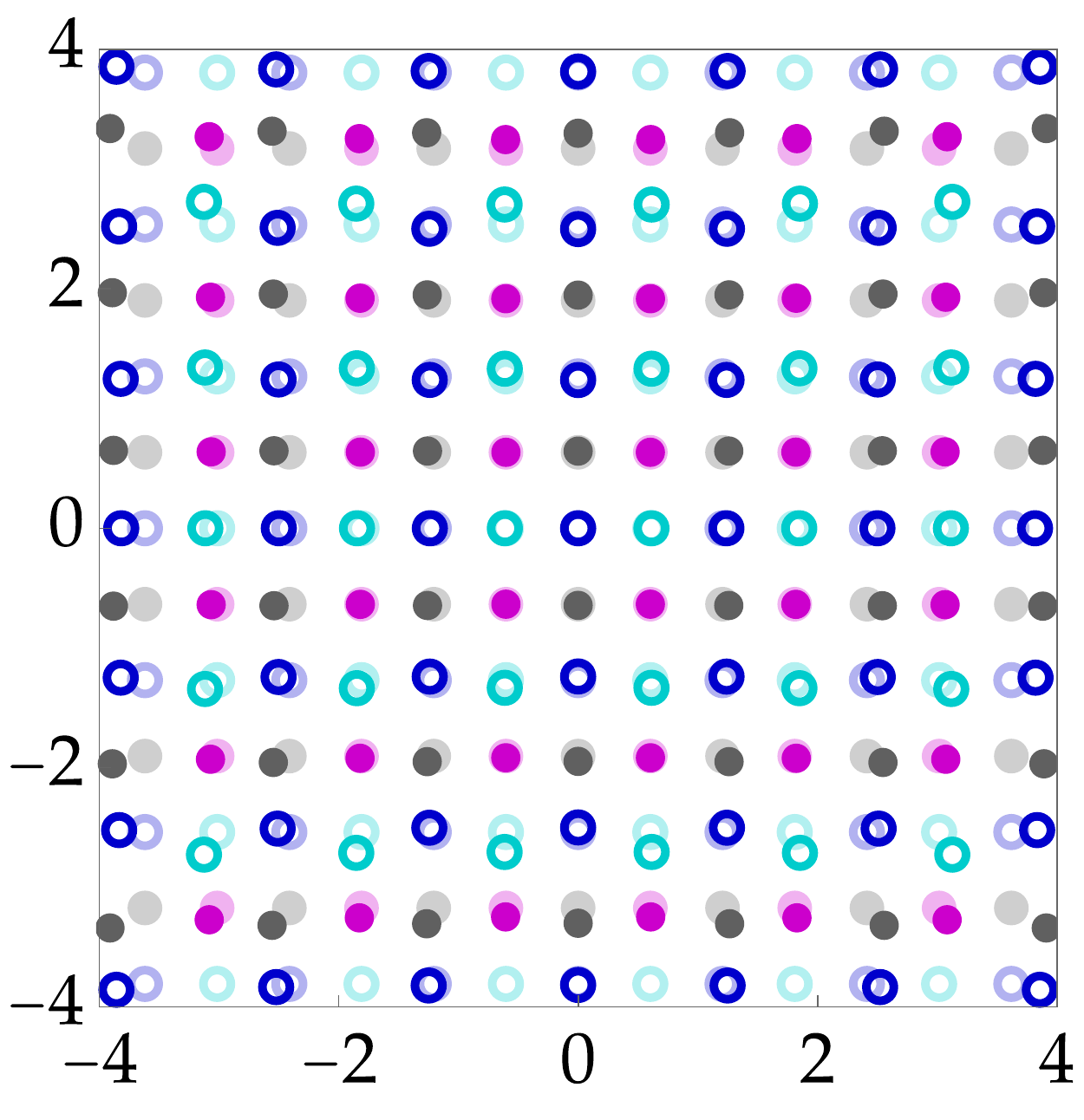}
\caption{$(m,n)=(9,8)$}
\end{subfigure}%
\begin{subfigure}{1.5 in} 
\includegraphics[width=1.5in]{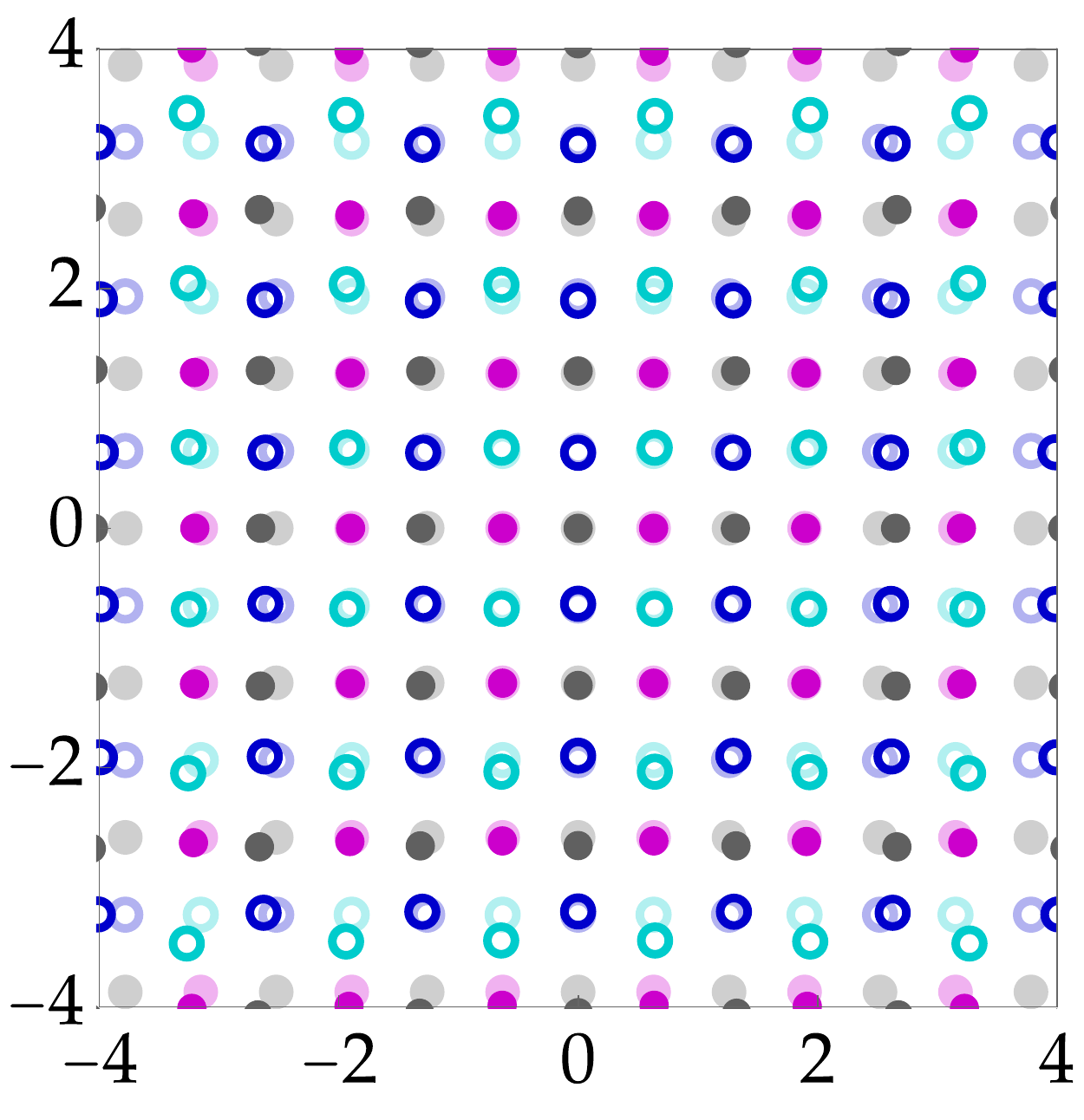}
\caption{$(m,n)=(9,9)$}
\end{subfigure}
\end{center}
\caption{Zeros and poles of $u^{[1]}_\mathrm{gH}(T^{-\frac{1}{2}}\zeta;m,n)$ in the $\zeta$-plane (bold colors) and their large-$T$ elliptic function approximations (faded colors).  Cyan circles:  zeros with positive derivative.  Blue circles:  zeros with negative derivative.  Magenta dots:  poles with positive residue.  Grey dots:  poles with negative residue.  For the point at $\zeta=0$ which corresponds to $x=0$, these match the theoretical predictions of Proposition~\ref{prop:BehaviorAtOrigin}.}
\label{fig:gH-poles-zeros-type-1}
\end{figure}
\begin{figure}
\begin{center}
\begin{subfigure}{1.5in}
\includegraphics[width=1.5in]{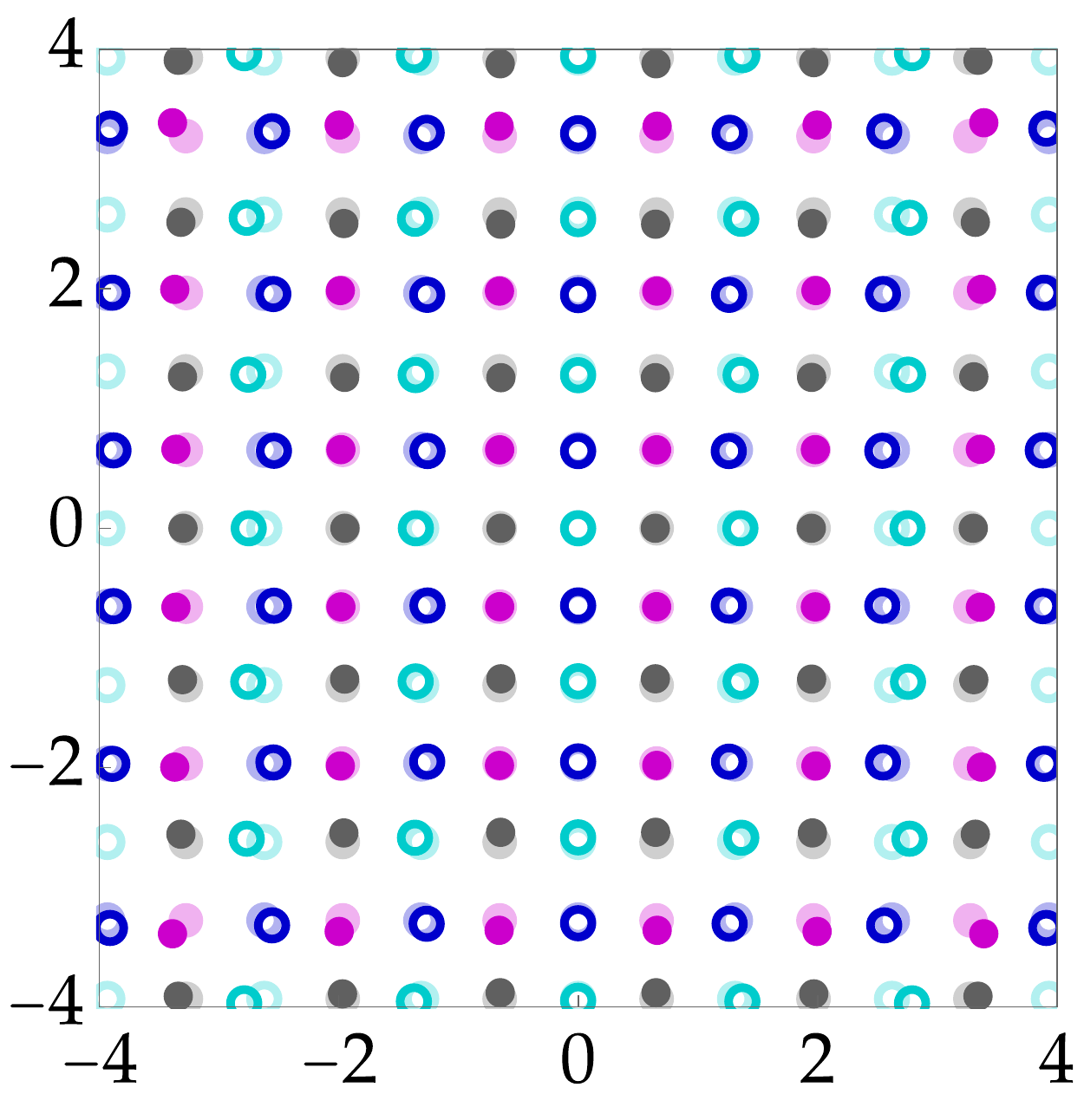}
\caption{$(m,n)=(8,8)$}
\end{subfigure}%
\begin{subfigure}{1.5 in}
\includegraphics[width=1.5in]{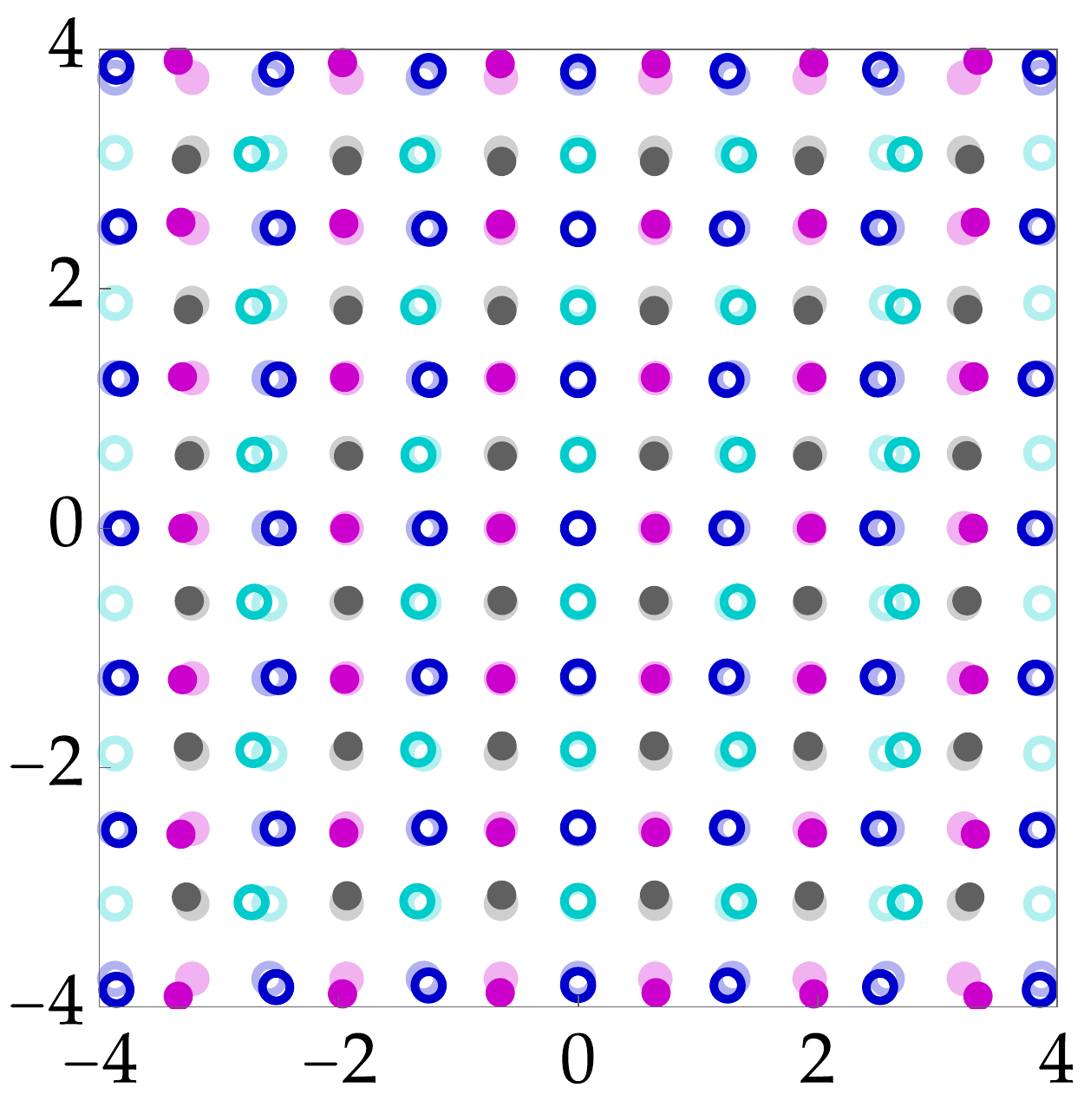}
\caption{$(m,n)=(8,9)$}
\end{subfigure}%
\begin{subfigure}{1.5in}
\includegraphics[width=1.5in]{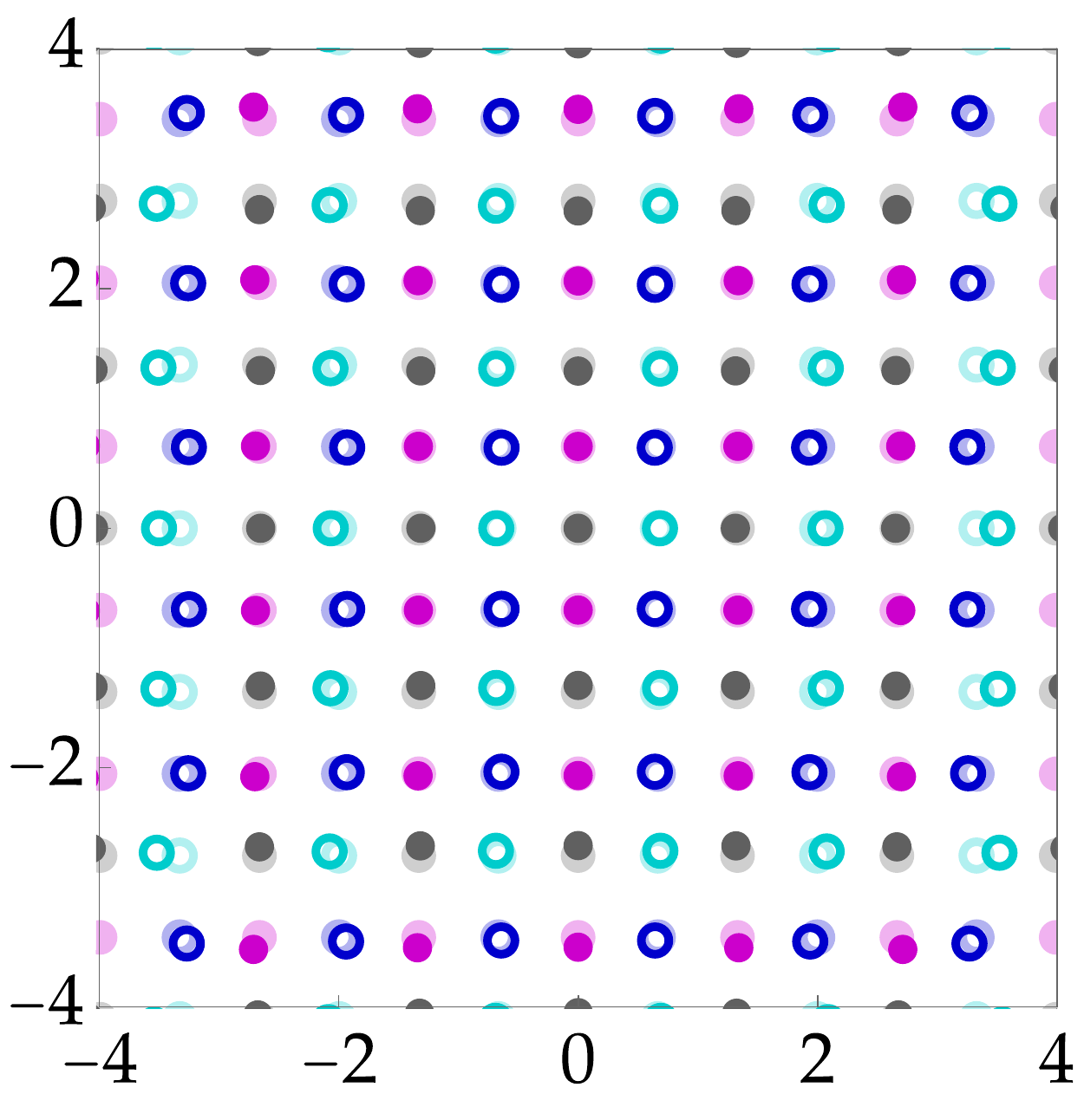}
\caption{$(m,n)=(9,8)$}
\end{subfigure}%
\begin{subfigure}{1.5in}
\includegraphics[width=1.5in]{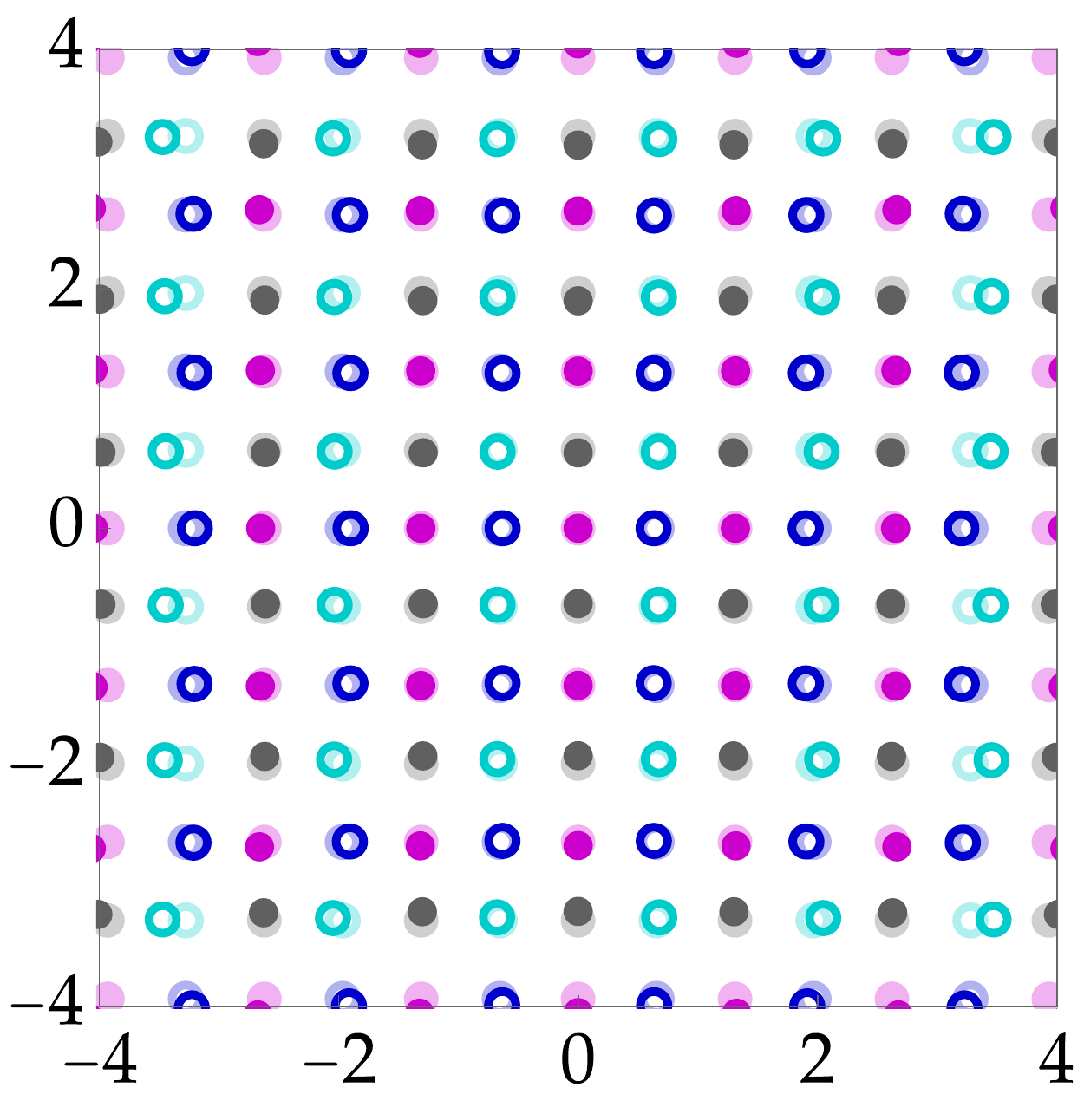}
\caption{$(m,n)=(9,9)$}
\end{subfigure}
\end{center}
\caption{As in Figure~\ref{fig:gH-poles-zeros-type-1} but for $u^{[2]}_\mathrm{gH}(T^{-\frac{1}{2}}\zeta;m,n)$.}
\label{fig:gH-poles-zeros-type-2}
\end{figure}
\begin{figure}
\begin{center}
\begin{subfigure}{1.5in}
\includegraphics[width=1.5in]{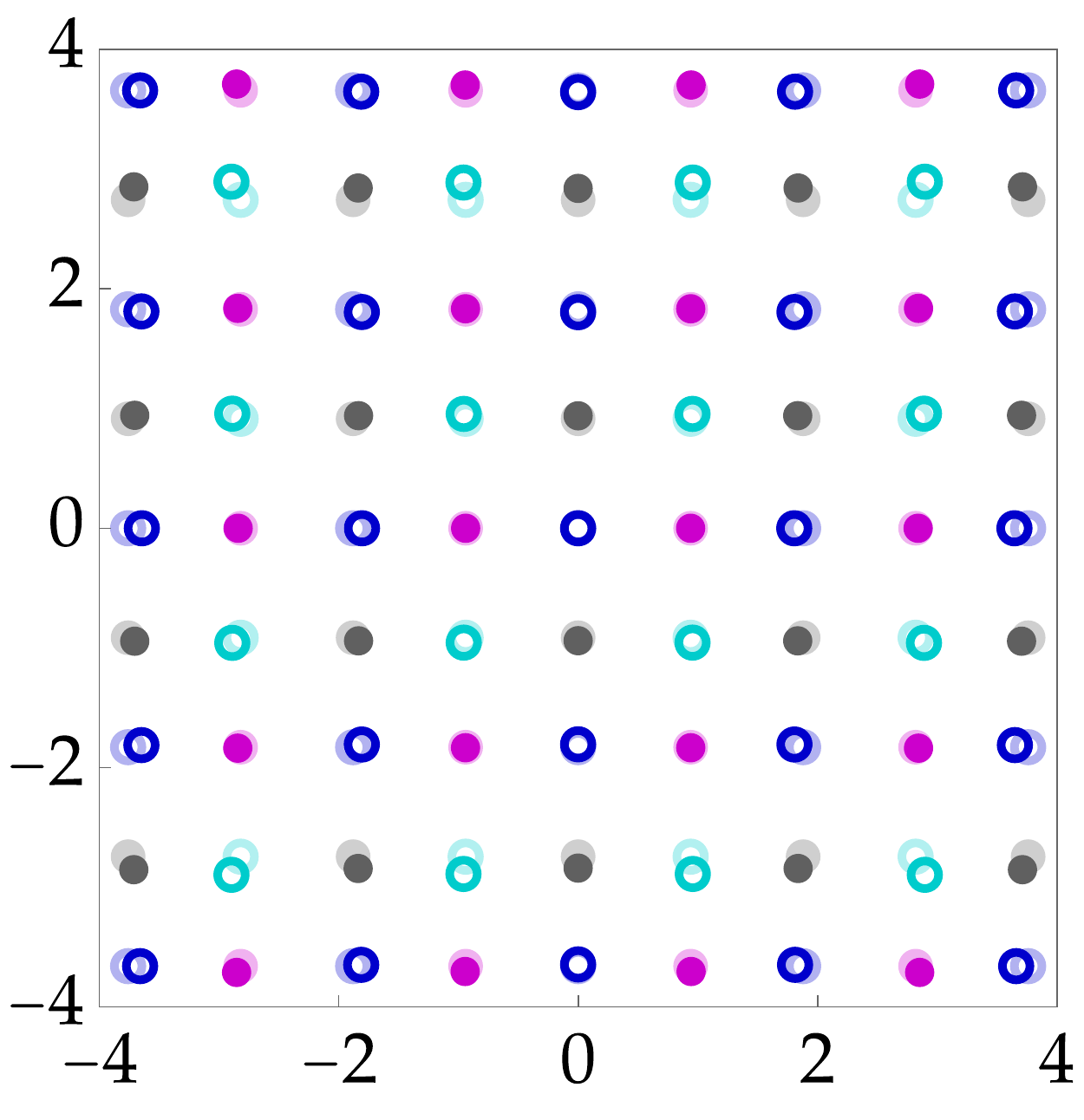}
\caption{$(m,n)=(8,8)$}
\end{subfigure}%
\begin{subfigure}{1.5 in}
\includegraphics[width=1.5in]{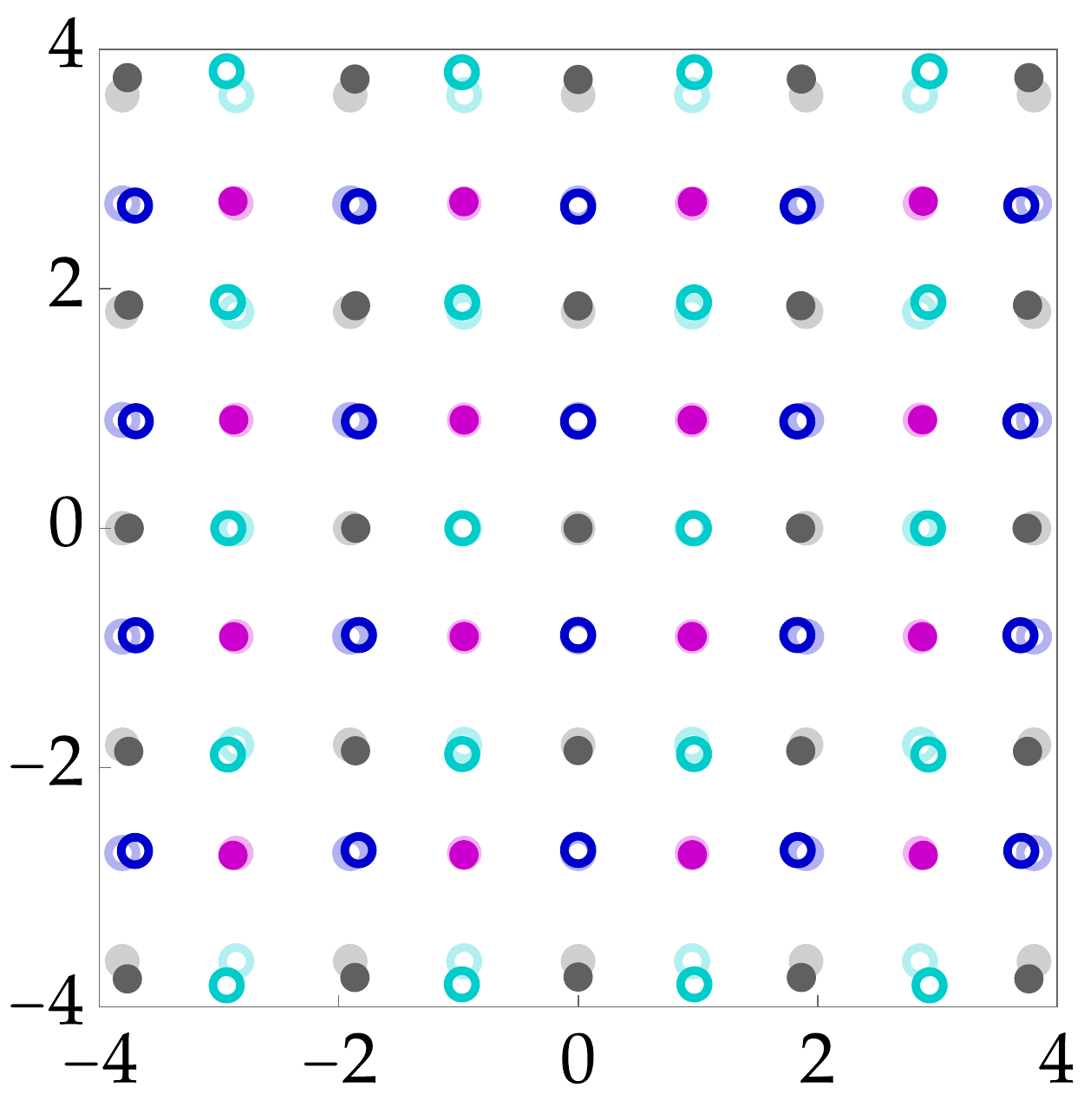}
\caption{$(m,n)=(8,9)$}
\end{subfigure}%
\begin{subfigure}{1.5in}
\includegraphics[width=1.5in]{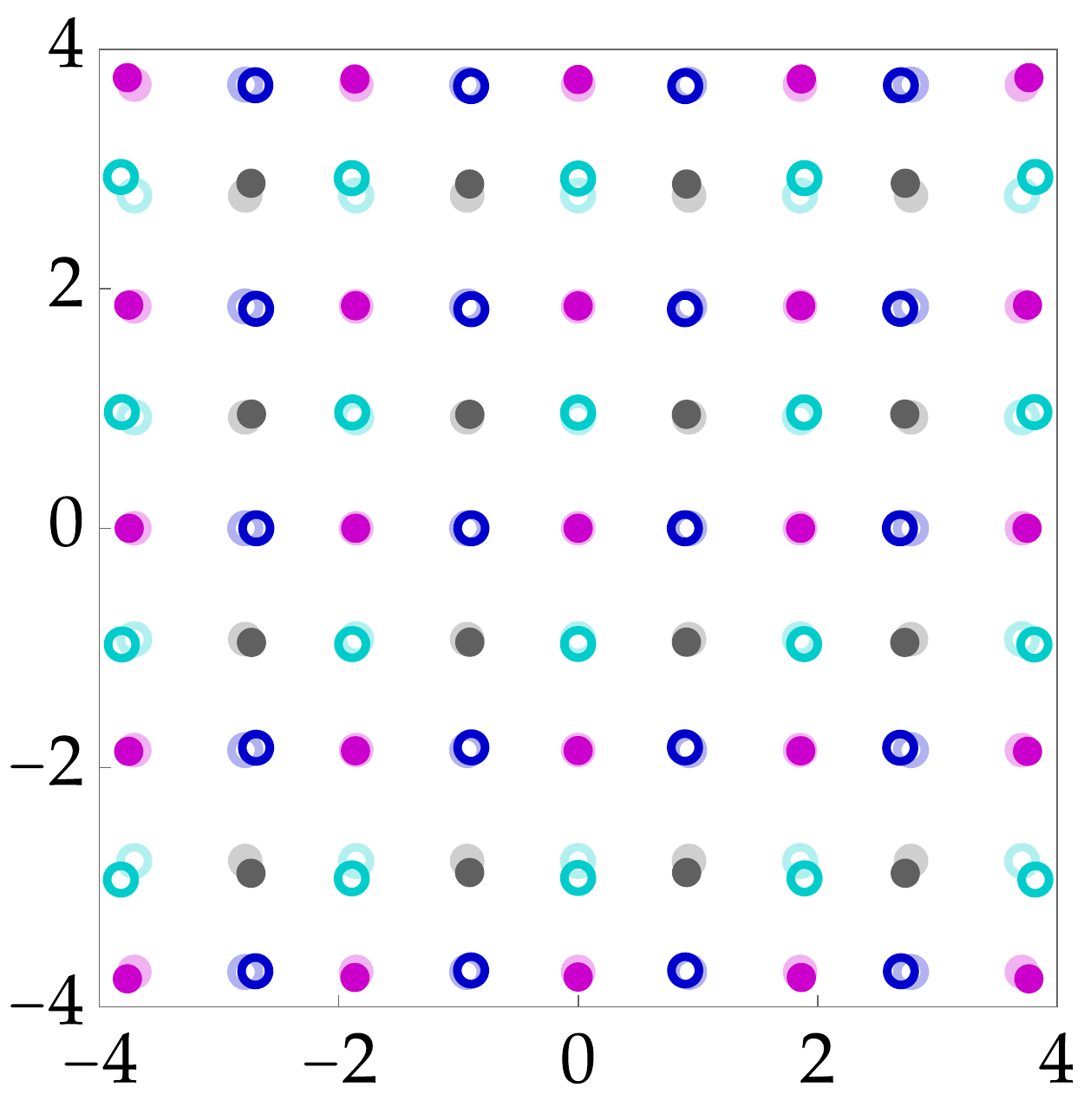}
\caption{$(m,n)=(9,8)$}
\end{subfigure}%
\begin{subfigure}{1.5in}
\includegraphics[width=1.5in]{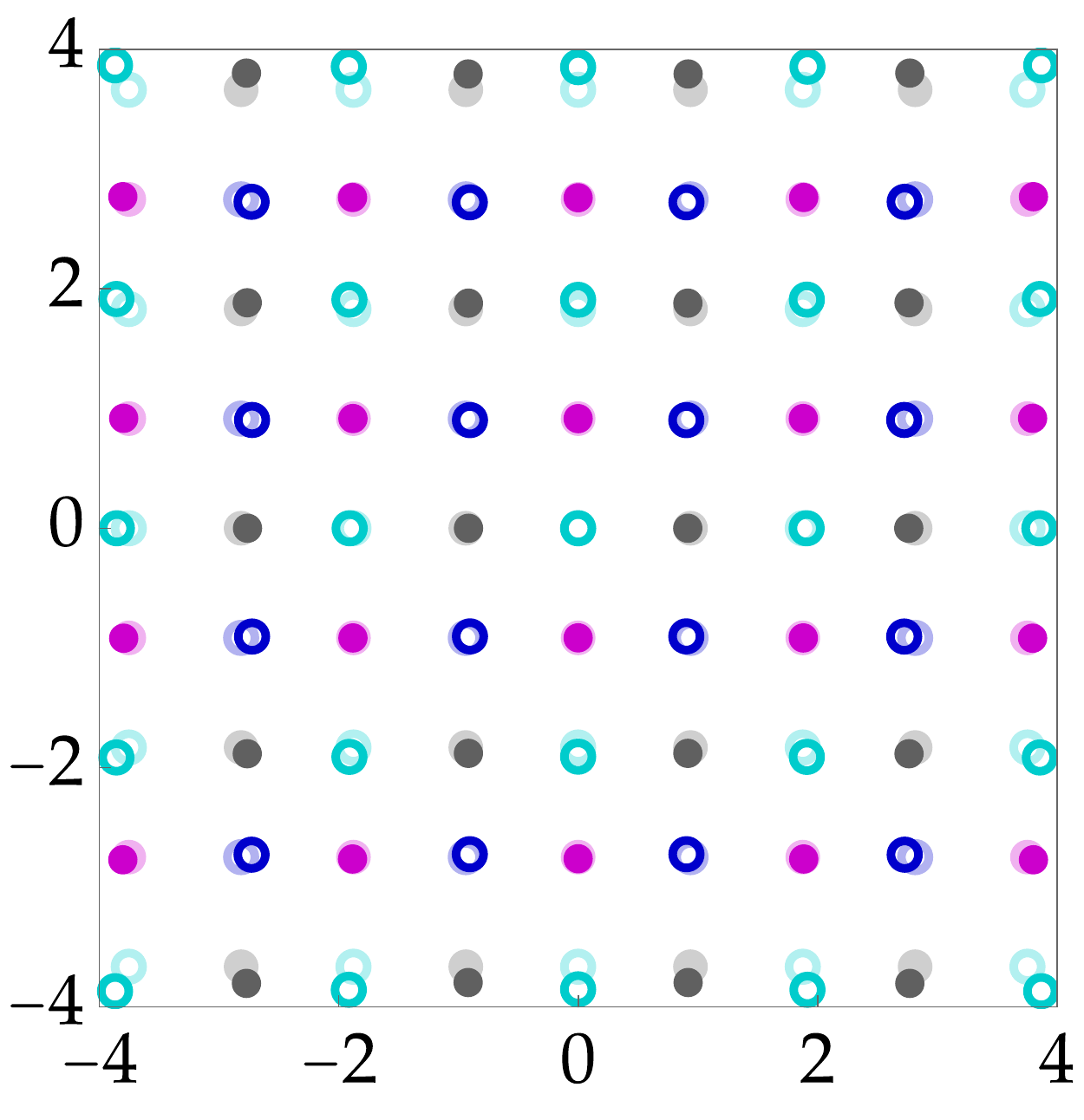}
\caption{$(m,n)=(9,9)$}
\end{subfigure}
\end{center}
\caption{As in Figures~\ref{fig:gH-poles-zeros-type-1} and \ref{fig:gH-poles-zeros-type-2} but for $u^{[3]}_\mathrm{gH}(T^{-\frac{1}{2}}\zeta;m,n)$.}
\label{fig:gH-poles-zeros-type-3}
\end{figure}

With just a bit more work, the analysis behind Corollary~\ref{cor:poles-and-zeros} allows one to extract the asymptotic behavior of the zeros of the special gH and gO polynomials themselves.  
For a uniform treatment of both families of polynomials, set $Q_\mathrm{gH}(x;m,n):=H_{m,n}(x)$ and $Q_\mathrm{gO}(x;m,n)=Q_{m,n}(x)$.  Then define
\eq
\mathcal{D}_\mathrm{F}(m,n):=\{y_0\in\mathbb{C}: Q_\mathrm{F}(|\Theta_{0,\mathrm{F}}^{[1]}(m,n)|^\frac{1}{2}y_0;m,n)=0\},\quad \mathrm{F}=\mathrm{gH},\mathrm{gO}
\endeq
as the set of all roots of the indicated polynomial, suitably rescaled.  Similarly, let $\dot{\mathcal{D}}_\mathrm{F}(m,n)$ denote the set of values of $y_0$ for which both conditions in \eqref{eq:intro-poles-quantize} hold with $\zeta=0$ fixed for type $j=1$ and $k=2$.  The selection of $k=2$ turns out to correspond to $\zeta=0$ being a pole of $f(\zeta-\zeta_0)$ of residue $-1$.
The following result was first proved in the gH case by Masoero and Roffelsen using the theory of a  family of quantum oscillators with anharmonic potentials.  The same authors are currently working on an analogous result for the gO polynomials \cite{MasoeroR20}.  Our proof is a consequence of the isomonodromy method and hence applies equally well in the gH and gO cases.

\begin{corollary}[Roots of gH and gO polynomials; cf.\@ \protect{\cite[Theorem 2]{MasoeroR19}} for the gH case]
Fix a rational aspect ratio $\rho>0$ and a compact set $C$ within one of the domains $\rectangle(\kappa)$, $\TR(\kappa)$, or $\TI(\kappa)$ (the latter two for the gO family only).  
Then, there is a constant $r>0$ (depending on $C$ and $\rho$) such that for $m,n$ sufficiently large with $n=\rho m$ the following statements hold with $T=|\Theta_{0,\mathrm{F}}^{[1]}(m,n)|$.  For each point $\dot{y}_0\in\dot{\mathcal{D}}_\mathrm{F}(m,n)\cap C$ there is a unique point $y_0\in\mathcal{D}_\mathrm{F}(m,n)$ that satisfies $|y_0-\dot{y}_0|\le rT^{-2}$.  Likewise for each point $y_0\in\mathcal{D}_\mathrm{F}(m,n)\cap C$ there is a unique point $\dot{y}_0\in\dot{\mathcal{D}}_\mathrm{F}(m,n)$ that satisfies $|\dot{y}_0-y_0|\le rT^{-2}$.
%
%If $C\subset\rectangle(\kappa)$ is compact and $C_\mathrm{gH}(m,n)\subset C$ is the set of $y_0\in C$ for which both conditions in \eqref{eq:intro-poles-quantize} hold with $\zeta=0$ fixed for the type $1$ gH family and $k=2$, then there is a constant $r>0$ (depending on $C$ and $\rho$) such that 
%for each point $y_0\in C_\mathrm{gH}(m,n)$ there is a unique simple zero $x$ of the polynomial $H_{m,n}(\cdot)$ that after rescaling by $y=T^{-\frac{1}{2}}x$ satisfies $|y-y_0|\le rT^{-2}$, and no other zeros $x$ with $y=T^{-\frac{1}{2}}x\in C$.
%
%If $C$ is a compact subset of 
%$\rectangle(\kappa)$, $\TR(\kappa)$, or $\TI(\kappa)$ and $C_\mathrm{gO}(m,n)\subset C$ is the set of $y_0\in C$ for which both conditions in \eqref{eq:intro-poles-quantize} hold with $\zeta=0$ fixed for the type $1$ gO family and $k=2$,  
%then there is a constant $r>0$ (depending on $C$ and $\rho$) such that for 
%for each point $y_0\in C_\mathrm{gO}(m,n)$ there is a unique simple zero $x$ of the polynomial $Q_{m,n}(\cdot)$ that after rescaling by $y=T^{-\frac{1}{2}}x$ satisfies $|y-y_0|\le rT^{-2}$, and no other zeros $x$ with $y=T^{-\frac{1}{2}}x\in C$.
\label{cor:polynomial-zeros}
\end{corollary}
The proof is given at the end of Section~\ref{sec:MalgrangeResidues}, and it is based on the fact that, according to Tables~\ref{tab:gH} and \ref{tab:gO}, the gH and gO polynomials $H_{m,n}(x)$ and $Q_{m,n}(x)$ give rise to poles of residue $-1$ of $u_\mathrm{gH}^{[1]}(x;m,n)$ and $u_\mathrm{gO}^{[1]}(x;m,n)$ respectively.  Therefore, the accuracy of the approximation of roots of the gH and gO polynomials given in Corollary~\ref{cor:polynomial-zeros} can also be seen in the bottom row of the first two columns of Figures~\ref{fig:gH-poles-zeros}--\ref{fig:gOnn-poles-zeros}.  
%The microscopic distribution of zeros of $H_{m,n}(x)$ and $Q_{m,n}(x)$ can be compared when $y=T^{-\frac{1}{2}}x\in\rectangle(\kappa)$.  Referring to Table~\ref{tab:outer-uniformize-phases}, one sees that for $(m,n)$ large, the zeros of $H_{m,n}(x)$ near a given point in $\rectangle(\kappa)$ are locally displaced from those of $Q_{m,n}(x)$ by about one sixth of a period in each direction ($Z_\mathfrak{a}$ and $Z_\mathfrak{b}$).  \textcolor{red}{This observation comes from subtracting the phases for the same $y_0$ and $\kappa$.  But $\kappa$ cannot be exactly the same value for the gH and gO families for the same $(m,n)$.  So probably this remark should be omitted.  What is actually the relationship between the zero lattices?}
%

All of our results are obtained by rigorous asymptotic analysis of Riemann-Hilbert representations of the Painlev\'e-IV rational solutions (see Riemann-Hilbert Problem~\ref{rhp:general} and Theorems~\ref{thm:gO-RHP} and \ref{thm:gH-RHP} below).  We next describe these auxiliary results and the methodology behind them.

\subsection{Riemann-Hilbert representations of rational Painlev\'e-IV solutions}
\label{sec:intro-Methodology}
%\subsection{Integrable approach to the generalized Okamoto rationals}
\subsubsection{A Lax pair for Painlev\'e-IV}
\label{sec:Lax-Pair}
%Let us rewrite the Painlev\'e-IV equation \eqref{p4} for a function $u=u(x)$ in the form used by Jimbo and Miwa \cite{Jimbo:1981a}:
%\eq
%u_{xx}=\frac{(u_x)^2}{2u}+\frac{3}{2}u^3+4xu^2+(2+2x^2-4\Theta_\infty)u-\frac{8\Theta_0^2}{u}, \quad u:\mathbb{C}\to\mathbb{C} \text{ with parameters } \Theta_0,\Theta_\infty\in\mathbb{C}.
%\label{p4}
%\endeq
%Comparing \eqref{p4} with \eqref{p4}, we see the constants $\Theta_0$ and $\Theta_\infty$ satisfy
%\eq
%-8\Theta_0^2=\beta, \quad 2\Theta_\infty-1=\alpha.
%\label{eq:Thetas-alpha-beta}
%\endeq
If $y(x)$ is a nonzero solution of
\eq
y'(x)=-(u(x)+2x)y(x)
\label{eq:yODE}
\endeq
and $z$ is defined uniquely in terms of $u$ by 
\eq
4z(x)=-u'(x) + u(x)^2 + 2xu(x) + 4\Theta_0,
\label{eq:LaxPair-z-define}
\endeq
then the differential equations \eqref{p4} and \eqref{eq:yODE} for $u$ and $y$ are the compatibility conditions for the 
Garnier-Jimbo-Miwa Lax pair \cite{Jimbo:1981a,FokasIKN:2006}
\eq
{\bf\Psi}_\lambda = {\bf \Lambda}{\bf\Psi}, \quad {\bf\Psi}_x = {\bf X}{\bf\Psi},
\label{eq:PIV-Lax-Pair}
\endeq
with coefficient matrices
\eq
\begin{split}
\mathbf{\Lambda}:=\lambda\sigma_3 + \mathbf{\Lambda}_0(x) + \lambda^{-1}\mathbf{\Lambda}_1(x),\quad
\mathbf{\Lambda}_0(x)&:=\bpm x & y(x)\\ 2y(x)^{-1}(z(x)-\Theta_0-\Theta_\infty) & -x\epm,\\
\mathbf{\Lambda}_1(x)&:=\bpm \Theta_0-z(x) & -\tfrac{1}{2}u(x)y(x)\\
2y(x)^{-1}u(x)^{-1}z(x)(z(x)-2\Theta_0) & z(x)-\Theta_0\epm
\end{split}
\label{eq:JM-A}
\endeq
%\eq
%{\bf A}:=\left(\lambda+x+\frac{1}{\lambda}(\Theta_0-z)\right)\sigma_3 + y\left(1-\frac{u}{2\lambda}\right)\sigma_+ + \frac{2}{y}\left(z-\Theta_0-\Theta_\infty+\frac{z}{\lambda u}(z-2\Theta_0)\right)\sigma_-
%\label{eq:JM-A}
%\endeq
and 
\eq
\mathbf{X}:=\lambda\sigma_3 + \mathbf{X}_0(x),\quad\mathbf{X}_0(x):=\bpm
0 & y(x)\\
2y(x)^{-1}(z(x)-\Theta_0-\Theta_\infty) & 0\epm = \mathbf{\Lambda}_0(x)-x\sigma_3.
\label{eq:JM-U}
\endeq
%\eq
%{\bf U}:=\lambda\sigma_3+y\sigma_+ + \frac{2}{y}(z-\Theta_0-\Theta_\infty)\sigma_-.
%\label{eq:JM-U}
%\endeq
%Here we have used 
%\eq
%\sigma_3:=\bpm 1 & 0 \\ 0 & -1 \epm, \quad \sigma_+:=\bpm 0 & 1 \\ 0 & 0 \epm, \quad \sigma_-:=\bpm 0 & 0 \\ 1 & 0 \epm.
%\endeq
Note that $\lambda=0$ is a Fuchsian (regular singular) point for the equation $\mathbf{\Psi}_\lambda=\mathbf{\Lambda\Psi}$, with exponents $\pm\Theta_0$.   
%In the case that \eqref{p4} admits a rational solution of gO type, it is easy to see that the singularity at $\lambda=0$ is nonresonant, i.e., the difference of the exponents is not an integer.  (By contrast, in the gH type 3 rational family, the singularity is resonant, but apparent.  In other words there exist two linearly independent Frobenius series solutions, or equivalently, the monodromy matrix about $\lambda=0$ is a multiple of the identity.)  
The only other singularity of $\mathbf{\Psi}_\lambda=\mathbf{\Lambda\Psi}$ is $\lambda=\infty$, an irregular singular point.  Formal expansions of solutions about $\lambda=\infty$ include a single-valued exponential factor and a sub-dominant factor proportional to $\lambda^{\pm\Theta_\infty}$.  Hence the utility of the Jimbo-Miwa parameters $(\Theta_0,\Theta_\infty)$ over other parameters common in the literature such as $(\alpha,\beta)=(2\Theta_\infty-1,-8\Theta_0^2)$ (see \cite[Eqn.\@ 32.2.4]{DLMF}) is that they explicitly encode the formal monodromy about $\lambda=0$ and $\lambda=\infty$ in the solution $\mathbf{\Psi}$ of the Lax system \eqref{eq:PIV-Lax-Pair}. 

\subsubsection{Rational solutions via isomonodromy theory of the Lax pair}
\label{sec:rational-isomonodromy}
Our approach to representing the rational solutions of \eqref{p4} in a form convenient for asymptotic analysis in the limit that the parameters $(\Theta_0,\Theta_\infty)$ are large consists of the following steps:
\begin{enumerate}
\item Select a family of rational solutions and isolate within that family a distinguished parameter pair $(\Theta_0,\Theta_\infty)$ and its corresponding unique rational solution to serve as a ``seed''.  Using the seed in the matrices defined by \eqref{eq:JM-A}--\eqref{eq:JM-U}, the Lax pair equations \eqref{eq:PIV-Lax-Pair} become compatible and admit simultaneous solutions.
\item Sow the seed, i.e., 
\begin{enumerate}
\item Find simultaneous fundamental solution matrices $\mathbf{\Psi}(\lambda,x)$ of the Lax pair suitably normalized for $\lambda$ in the different Stokes sectors near each irregular singular point of the ``spectral equation'' $\mathbf{\Psi}_\lambda = \mathbf{\Lambda\Psi}$ and in a full neighborhood of each regular singular point (Fuchsian singularity) of the same equation.  Compute explicitly the constants expressing the columns of each of these fundamental matrices as suitable linear combinations of the columns of the fundamental matrices for neighboring regions of the $\lambda$-plane (direct monodromy problem).  
\item Use this information to recast the fundamental matrices equivalently in terms of the solution of a matrix Riemann-Hilbert problem (inverse monodromy problem).
\end{enumerate}
In this step, we take full advantage of choice of seed solution to simplify the equation $\mathbf{\Psi}_x=\mathbf{X\Psi}$ and leverage this to obtain simultaneous solutions of \eqref{eq:PIV-Lax-Pair}.  This is in contrast to the usual approach in using the Lax pair \eqref{eq:PIV-Lax-Pair} to solve the initial-value problem for Painlev\'e equations, where only initial values are available and therefore one must instead start by solving the more complicated spectral equation $\mathbf{\Psi}_\lambda=\mathbf{\Lambda\Psi}$.
\item Reap the harvest, i.e., 
\begin{enumerate}
\item Apply isomonodromic Schlesinger transformations to increment/decrement the integer parameters of $(\Theta_0,\Theta_\infty)$, and hence obtain a Riemann-Hilbert problem for each pair $(\Theta_0,\Theta_\infty)$ in a certain lattice (for which the Schlesinger transformations are well-defined).
\item Show that the resulting lattice matches the full family of parameters for the rational solution family from which the seed was selected, and that the Riemann-Hilbert problem for given parameters in the family encodes a rational solution of the Painlev\'e equation at hand.
\end{enumerate}
\end{enumerate}
This method is general, and it has been applied before to characterize the rational solutions of the Painlev\'e-II equation \cite{BuckinghamM12,MillerS:2017}, the rational solutions of the Painlev\'e-III equation \cite{BothnerMS18}, and the gO rational solutions for the Painlev\'e-IV equation \cite{NovokshenovS14} (although the Riemann-Hilbert problem reported in that paper differs from the one we shall develop below).  The isomonodromy approach avoids completely the need for special determinantal representations of rational solutions having suitable analytic structure as has been used to study the rational solutions of Painlev\'e-II \cite{BertolaB:2015} and the gH rational solutions of Painlev\'e-IV \cite{Buckingham18}.  Hence it is useful in the study of rational solutions that are not known to have such representations, such as the rational solutions of Painlev\'e-III and the gO rational solutions of Painlev\'e-IV.  Even though such a determinantal representation is available for the gH rational solutions, the isomonodromy approach allows the gH and gO rational solutions of Painlev\'e-IV to be analyzed more-or-less on the same footing, which is a main point of our paper.

Now we give some more details about how the method applies to the Painlev\'e-IV equation, for which
the spectral equation $\mathbf{\Psi}_\lambda=\mathbf{\Lambda\Psi}$ has an irregular singular point at $\lambda=\infty$ and a regular singular point at $\lambda=0$ with Frobenius exponents $\pm\Theta_0$.  The irregular singular point has four Stokes sectors that we will label as (following the subscript notation of \cite{FokasIKN:2006}):
\eq
\begin{split}
S_1&:=\{\lambda\in\mathbb{C},\,\,\lambda\neq 0,\,\,-\tfrac{1}{2}\pi\le\arg(\lambda)\le 0\}\\
S_2&:=\{\lambda\in\mathbb{C},\,\,\lambda\neq 0,\,\,0\le\arg(\lambda)\le\tfrac{1}{2}\pi\}\\
S_3&:=\{\lambda\in\mathbb{C},\,\,\lambda\neq 0,\,\,\tfrac{1}{2}\pi\le\arg(\lambda)\le\pi\}\\
S_4&:=\{\lambda\in\mathbb{C},\,\,\lambda\neq 0,\,\,-\pi\le\arg(\lambda)\le -\tfrac{1}{2}\pi\}.
\end{split}
\endeq
Associated with each Stokes sector $S_j$, there is a simultaneous fundamental solution matrix $\mathbf{\Psi}=\mathbf{\Psi}^{(\infty)}_j(\lambda,x)$ of both equations of the Lax pair \eqref{eq:PIV-Lax-Pair} determined by the normalization condition
\eq
\mathbf{\Psi}_j^{(\infty)}(\lambda,x)\lambda^{\Theta_\infty\sigma_3}\mathrm{e}^{-(\frac{1}{2}\lambda^2+x\lambda)\sigma_3} = \mathbb{I}+\bo(\lambda^{-1}),\quad\lambda\to\infty,\quad \lambda\in S_j,\quad j=1,2,3,4,
\label{eq:infty-asymp}
\endeq
where the power functions $\lambda^{\pm\Theta_\infty}$ refer to the principal branches.
Applying Abel's Theorem to the simultaneous equations \eqref{eq:PIV-Lax-Pair} noting that $\mathrm{tr}(\mathbf{\Lambda})=\mathrm{tr}(\mathbf{X})=0$, it follows that the four solutions satisfy $\det(\mathbf{\Psi}_j^{(\infty)}(\lambda,x))=1$, $j=1,\dots,4$.

For simultaneous solutions of \eqref{eq:PIV-Lax-Pair} near $\lambda=0$, observe that whenever $(\Theta_0,\Theta_\infty)\in\Lambda_\mathrm{gO}$, the Frobenius exponents $\pm\Theta_0$ are unequal mod $\mathbb{Z}$, making the regular singular point \emph{nonresonant} and guaranteeing the existence of a basis of convergent Puiseux series solutions that can be found by the method of Frobenius.  On the other hand, when $(\Theta_0,\Theta_\infty)\in\Lambda_\mathrm{gH}$, the exponents always differ by integers making the singular point \emph{resonant}.  In general, the method of Frobenius fails to produce a basis of solutions near a resonant regular singular point, however we will see that such a basis indeed exists nonetheless when $(\Theta_0,\Theta_\infty)\in\Lambda_\mathrm{gH}$ and the coefficients in the Lax pair refer to the corresponding rational solution, making the resonant singular point an \emph{apparent singularity}.  Whether the singularity is nonresonant, or resonant but apparent, there exists a fundamental simultaneous solution matrix $\mathbf{\Psi}=\mathbf{\Psi}^{(0)}(\lambda,x)$ defined for $\lambda$ on a neighborhood of $\lambda=0$ with a branch cut on the negative real line omitted, such that
\eq
\mathbf{\Psi}^{(0)}(\lambda,x)\lambda^{-\Theta_0\sigma_3} \;\text{is analytic at $\lambda=0$}
\label{eq:zero-asymp}
\endeq
(and hence entire, since there are no other finite singular points) where the power functions $\lambda^{\pm\Theta_0}$ indicate principal branches. In the nonresonant case, $\mathbf{\Psi}^{(0)}(\lambda,x)$ is unique up to multiplication on the right by a constant invertible diagonal matrix, while in the resonant but apparent case there is additional freedom that enters via the ambiguity of adding an arbitrary multiple of the subdominant solution to the dominant solution.  Since Abel's Theorem implies that $\mathbf{\Psi}^{(0)}(\lambda,x)$ has constant determinant, we agree to partly resolve the ambiguity in this solution by insisting that $\det(\mathbf{\Psi}^{(0)}(\lambda,x))=1$.

The five simultaneous fundamental solutions are necessarily related pairwise on certain overlap domains by right-multiplication by constant matrices.  In particular, the following constant matrices are well-defined and have unit determinants:
\eq
\mathbf{V}_{2,1}:=\mathbf{\Psi}^{(\infty)}_1(\lambda,x)^{-1}\mathbf{\Psi}^{(\infty)}_2(\lambda,x),\quad \arg(\lambda)=0,
\label{eq:V21-def}
\endeq
\eq
\mathbf{V}_{2,3}:=\mathbf{\Psi}^{(\infty)}_3(\lambda,x)^{-1}\mathbf{\Psi}^{(\infty)}_2(\lambda,x),\quad
\arg(\lambda)=\tfrac{1}{2}\pi,
\label{eq:V23-def}
\endeq
\eq
\mathbf{V}_{4,3}:=\mathbf{\Psi}^{(\infty)}_3(\lambda,x)^{-1}\mathbf{\Psi}^{(\infty)}_4(\lambda,x),\quad
\arg(-\lambda)=0,
\label{eq:V43-def}
\endeq
\eq
\mathbf{V}_{4,1}:=\mathbf{\Psi}^{(\infty)}_1(\lambda,x)^{-1}\mathbf{\Psi}^{(\infty)}_4(\lambda,x),\quad\arg(\lambda)=-\tfrac{1}{2}\pi,
\label{eq:V41-def}
\endeq
\eq
\mathbf{V}_j:=\mathbf{\Psi}^{(\infty)}_j(\lambda,x)^{-1}\mathbf{\Psi}^{(0)}(\lambda,x),\quad\lambda\in S_j,\quad j=1,3,
\label{eq:V1V3-def}
\endeq
and
\eq
\mathbf{V}_j:=\mathbf{\Psi}^{(0)}(\lambda,x)^{-1}\mathbf{\Psi}^{(\infty)}_j(\lambda,x),\quad\lambda\in S_j,\quad j=2,4.
\label{eq:V2V4-def}
\endeq
The \emph{Stokes matrices} $\mathbf{V}_{j,k}$ are necessarily triangular (upper for $\mathbf{V}_{2,3}$ and $\mathbf{V}_{4,1}$, lower for $\mathbf{V}_{2,1}$ and $\mathbf{V}_{4,3}$), and their off-diagonal elements are \emph{Stokes multipliers} measuring the Stokes phenomenon associated with the irregular singular point of $\mathbf{\Psi}_\lambda=\mathbf{\Lambda\Psi}$ at $\lambda=\infty$.  The remaining four matrices $\mathbf{V}_j$ are called \emph{connection matrices}.  These matrices are always related by the following identities:
\eq
\mathbf{V}_{2,1}=\mathbf{V}_1\mathbf{V}_2,\quad
\mathbf{V}_{2,3}=\mathbf{V}_3\mathbf{V}_2,\quad
\mathbf{V}_{4,1}=\mathbf{V}_1\mathbf{V}_4,\quad
\mathbf{V}_{4,3}=\mathbf{V}_3\ee^{-2\pi\ii\Theta_0\sigma_3}\mathbf{V}_4,\quad
\mathbf{V}_{2,3}\mathbf{V}_{2,1}^{-1}\mathbf{V}_{4,1}\mathbf{V}_{4,3}^{-1}=\ee^{2\pi\ii\Theta_\infty\sigma_3}.
\label{eq:consistency}
\endeq
Modulo these identities, the Stokes and connection matrices constitute the \emph{monodromy data} for the seed solution.  It turns out that the monodromy data is the same for all rational solutions in the gO family, and is the same for each type of rational solution in the gH family.

Assuming existence of all five particular simultaneous solutions for a given value of $x\in\mathbb{C}$, and assuming that the Fuchsian singularity at $\lambda=0$ is either nonresonant, or resonant but apparent, 
the matrix function $\mathbf{Y}(\lambda;x)$ defined as follows:
\eq
\mathbf{Y}(\lambda;x):=\begin{cases}
\mathbf{\Psi}^{(\infty)}_j(\lambda,x)\ee^{-(\frac{1}{2}\lambda^2+x\lambda)\sigma_3},\quad & \lambda\in S_j,\quad |\lambda|>1,\quad j=1,2,3,4,\\
\mathbf{\Psi}^{(0)}(\lambda,x)\ee^{-(\frac{1}{2}\lambda^2+x\lambda)\sigma_3},&\quad |\lambda|<1,
\end{cases}
\label{eq:Y-Psi}
\endeq
solves the following Riemann-Hilbert problem relative to the jump contour $\Sigma$ shown in Figure~\ref{fig:PIV-Sigma}.  
\begin{figure}[h]
\begin{center}
\includegraphics{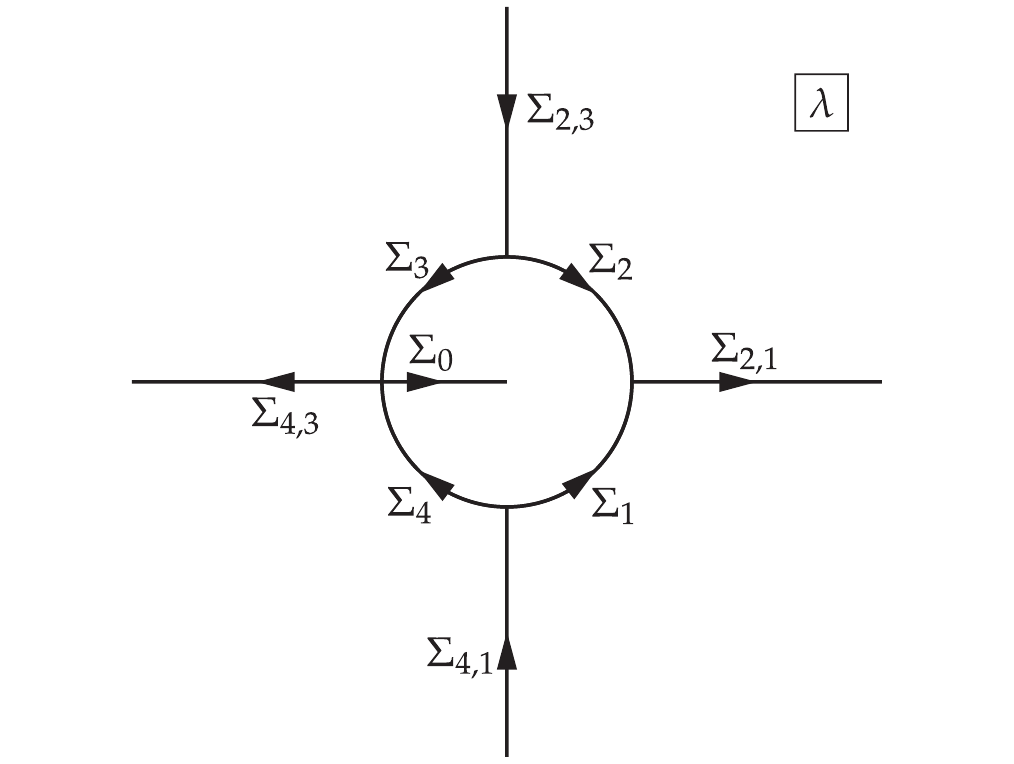}
\end{center}
\caption{The oriented contour $\Sigma$ consists of four rays, the oriented segment $\Sigma_0=(-1,0)$, and four oriented arcs of the unit circle in the $\lambda$-plane.}
\label{fig:PIV-Sigma}
\end{figure}
\begin{rhp}[Painlev\'e-IV Inverse Monodromy Problem]
Fix $(\Theta_0,\Theta_\infty)\in\mathbb{C}^2$.  Let Stokes matrices  $\mathbf{V}_{j,k}$, $(j,k)=(2,1), (2,3), (4,3), (4,1)$, and connection matrices $\mathbf{V}_j$, $j=1,\dots,4$ be given, and assume that they all have unit determinant, that the Stokes matrices have the correct triangular structure, and that the matrices are related by \eqref{eq:consistency}.
Seek a $2\times 2$ matrix function $\lambda\mapsto \mathbf{Y}(\lambda;x)$ with the following properties:
\begin{itemize}
\item\textbf{Analyticity:}  the function $\lambda\mapsto\mathbf{Y}(\lambda;x)$ is analytic for $\lambda\in\mathbb{C}\setminus\Sigma$.
\item\textbf{Jump conditions:} $\mathbf{Y}(\lambda;x)$ assumes continuous boundary values on $\Sigma$ from each component of $\mathbb{C}\setminus\Sigma$, except at the origin.  Using a subscript $+$ (resp., $-$) to indicate a boundary value taken from the left (resp., right) by orientation, the boundary values are related on each arc of $\Sigma$ by the jump condition
\eq
\mathbf{Y}_+(\lambda;x)=\mathbf{Y}_-(\lambda;x)\ee^{(\frac{1}{2}\lambda^2+x\lambda)\sigma_3}\mathbf{V}\ee^{-(\frac{1}{2}\lambda^2 + x\lambda)\sigma_3}
\endeq
where $\mathbf{V}$ is the arcwise-constant function defined on $\Sigma$ as follows: 
\eq
\mathbf{V}:=\mathbf{V}_{j,k},\quad \lambda\in\Sigma_{j,k},\quad (j,k)=(2,1), (2,3), (4,3), (4,1),
\endeq
\eq
\mathbf{V}:=\mathbf{V}_j,\quad\lambda\in\Sigma_j,\quad j=1,\dots,4,
\endeq
and
\eq
\mathbf{V}:=\mathbf{V}_0=\ee^{2\pi\ii\Theta_0\sigma_3},\quad\lambda\in\Sigma_0.
\endeq
\item\textbf{Behavior near the origin:} $\mathbf{Y}(\lambda;x)\lambda^{-\Theta_0\sigma_3}$ is bounded as $\lambda\to 0$.
\item\textbf{Normalization:}  $\mathbf{Y}(\lambda;x)\lambda^{\Theta_\infty\sigma_3}\to\mathbb{I}$ as $\lambda\to\infty$.
\end{itemize}
\label{rhp:general}
\end{rhp}
Assuming that this problem has a unique solution, one can see that the various fundamental simultaneous solution matrices of the differential equations \eqref{eq:PIV-Lax-Pair} can be obtained instead from their monodromy data by solving Riemann-Hilbert Problem~\ref{rhp:general} and using \eqref{eq:Y-Psi}.

In Appendix~\ref{sec:Okamoto-RHP-basic-properties} we describe some basic properties of Riemann-Hilbert Problem~\ref{rhp:general}, and we discuss Schlesinger transformations that preserve the monodromy data and act on solutions of Painlev\'e-IV via induced B\"acklund transformations in Appendix~\ref{sec:Schlesinger}.

\subsubsection{Universal Riemann-Hilbert representations of rational solutions of Painlev\'e-IV}
\label{sec:universal-RHPs}
From a solution $\mathbf{Y}(\lambda;x)$ of Riemann-Hilbert Problem~\ref{rhp:general}, we can try to define two matrix functions of $x\in\mathbb{C}$ by 
\eq
\mathbf{Y}^0_0(x):=\lim_{\lambda\to 0}\mathbf{Y}(\lambda;x)\lambda^{-\Theta_0\sigma_3}\quad\text{and}\quad
\mathbf{Y}^\infty_1(x):=\lim_{\lambda\to\infty}\lambda(\mathbf{Y}(\lambda;x)\lambda^{\Theta_\infty\sigma_3}-\mathbb{I})
\label{eq:Z0-Y1-define}
\endeq
and related scalar functions given by 
\eq
u(x):=-2\Theta_0\frac{Y^0_{0,11}(x)Y^0_{0,12}(x)}{Y^\infty_{1,12}(x)}\quad\text{and}\quad
u_\tw(x):=-2\frac{Y^0_{0,21}(x)Y^\infty_{1,12}(x)}{Y^0_{0,11}(x)},
\label{eq:u-ucirc}
\endeq
wherever these definitions make sense.  The following theorem is proved in Appendix~\ref{sec:RHPgO} by implementing the method described at the beginning of Section~\ref{sec:rational-isomonodromy} on a suitable seed in the gO family.
\begin{theorem}[Riemann-Hilbert representation of gO rational solutions]
Fix $(\Theta_0,\Theta_\infty)\in\Lambda_\mathrm{gO}$. 
Let Stokes matrices $\mathbf{V}_{j,k}$ be defined by 
\eq
\mathbf{V}_{2,1}=\bpm 1 & 0\\2\ii & 1\epm,\quad
\mathbf{V}_{2,3}=
\bpm 1 & -\tfrac{1}{2}\ii\\0 & 1\epm,\quad
\mathbf{V}_{4,3}=\ee^{2\pi\ii\Theta_\infty}
\bpm 1 & 0\\2\ii & 1\epm,\quad\text{and}\quad
\mathbf{V}_{4,1}=\bpm 1 & -\tfrac{1}{2}\ii\\0 & 1\epm,
\label{eq:gO-Stokes}
\endeq
and let connection matrices $\mathbf{V}_j$ be defined by 
\eq
\mathbf{V}_1=\bpm \tfrac{1}{\sqrt{3}} & -\tfrac{1}{2}\\\tfrac{2}{\sqrt{3}}\ee^{\frac{\ii\pi}{6}} & \ee^{-\frac{\ii\pi}{6}}\epm,\quad
\mathbf{V}_2 = \bpm \ee^{\frac{\ii\pi}{6}} & \tfrac{1}{2}\\ \tfrac{2}{\sqrt{3}}\ee^{\frac{5\ii\pi}{6}} & \tfrac{1}{\sqrt{3}}\epm,\quad
\mathbf{V}_3 = \bpm \tfrac{1}{\sqrt{3}}\ee^{-\frac{\ii\pi}{3}} & \tfrac{1}{2}\ee^{-\frac{2\ii\pi}{3}}\\
\tfrac{2}{\sqrt{3}}\ee^{-\frac{\ii\pi}{6}} & \ee^{\frac{\ii\pi}{6}}\epm,\quad\text{and}\quad
\mathbf{V}_4 = \bpm \ee^{-\frac{\ii\pi}{6}} & \tfrac{1}{2}\ee^{-\frac{\ii\pi}{3}}\\
\tfrac{2}{\sqrt{3}}\ee^{-\frac{5\ii\pi}{6}} & \tfrac{1}{\sqrt{3}}\ee^{\frac{\ii\pi}{3}}\epm,
\label{eq:gO-connection}
\endeq
which satisfy the consistency conditions \eqref{eq:consistency}.  Then
Riemann-Hilbert Problem~\ref{rhp:general} has a unique solution for all but finitely many values of $x\in\mathbb{C}$, and the functions $u(x)$ and $u_\tw(x)$ defined by \eqref{eq:u-ucirc}
are the unique (gO) rational solutions of the Painlev\'e-IV equation \eqref{p4} for parameters $(\Theta_0,\Theta_\infty)$ and for $(\Theta_{0,\tw},\Theta_{\infty,\tw})\in\Lambda_\mathrm{gO}$ defined in \eqref{eq:Baecklund-3-to-1} respectively.  
\label{thm:gO-RHP}
\end{theorem}

Similarly, the following theorem is proved in Appendix~\ref{sec:RHPgH} by applying the method to a suitable seed in the type-$3$ sector of the gH family.
\begin{theorem}[Riemann-Hilbert representation of gH rational solutions]
Fix $(\Theta_0,\Theta_\infty)\in\Lambda_\mathrm{gH}^{[3]+}$.  Let Stokes matrices $\mathbf{V}_{j,k}$ be defined by 
\eq
\mathbf{V}_{2,1}=\mathbf{V}_{2,3}=\mathbf{V}_{4,1}=\mathbb{I},\quad\mathbf{V}_{4,3}=\ee^{2\pi\ii\Theta_\infty}\mathbb{I},
\label{eq:gH-Stokes}
\endeq
and let connection matrices $\mathbf{V}_j$ be defined by 
\eq
\mathbf{V}_1=\mathbf{V}_3=\bpm 1 & 0\\1 & 1\epm\quad\text{and}\quad
\mathbf{V}_2=\mathbf{V}_4=\bpm 1 & 0\\-1 & 1\epm,
\label{eq:gH-connection}
\endeq
which satisfy the consistency conditions \eqref{eq:consistency}.  Then Riemann-Hilbert Problem~\ref{rhp:general} has a unique solution for all but finitely many values of $x\in\mathbb{C}$, and the functions $u(x)$ and $u_\tw(x)$ defined by \eqref{eq:u-ucirc} are the unique (gH) rational solutions of the Painlev\'e-IV equation \eqref{p4} for parameters $(\Theta_0,\Theta_\infty)\in\Lambda_\mathrm{gH}^{[3]+}$ and for $(\Theta_{0,\tw},\Theta_{\infty,\tw})\in\Lambda_\mathrm{gH}^{[1]-}$ defined in \eqref{eq:Baecklund-3-to-1} respectively.
\label{thm:gH-RHP}
\end{theorem}
\begin{remark}
Since three of the Stokes matrices in \eqref{eq:gH-Stokes} are trivial and since the connection matrices in \eqref{eq:gH-connection} are either identical or related by matrix inversion on oppositely oriented arcs of the unit circle $|\lambda|=1$, in the case of gH rationals for $(\Theta_0,\Theta_\infty)\in\Lambda_\mathrm{gH}^{[3]+}$ we can use a simplified jump contour $\Sigma=\Sigma_\mathrm{gH}=S^1\cup\Sigma_0\cup\Sigma_{4,3}$ and jump matrix $\mathbf{Y}_-(\lambda;x)^{-1}\mathbf{Y}_+(\lambda;x)$ as shown in Figure~\ref{fig:PIV-Sigma-Hermite}.
\begin{figure}[h]
\begin{center}
\includegraphics{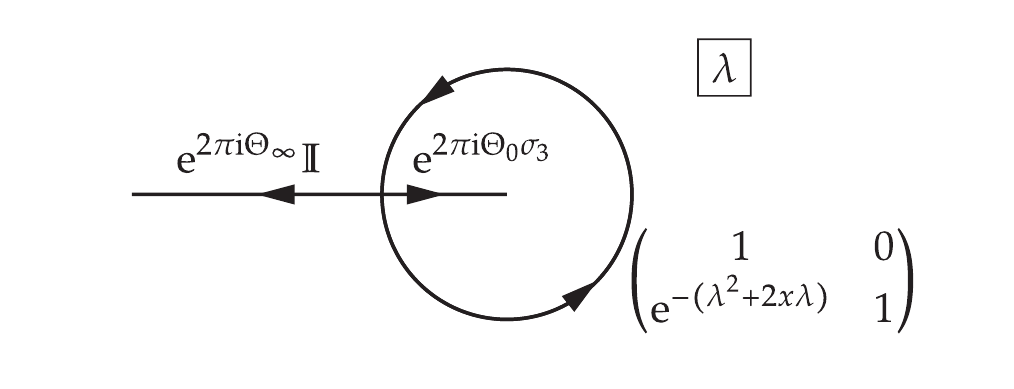}
\end{center}
\caption{The oriented contour $\Sigma_\mathrm{gH}=S^1\cup\Sigma_0\cup\Sigma_{4,3}$
and the jump matrix $\mathbf{Y}_-(\lambda;x)^{-1}\mathbf{Y}_+(\lambda;x)$ for gH rational solutions with $(\Theta_0,\Theta_\infty)\in\Lambda_\mathrm{gH}^{[3]+}$. 
%consists of one infinite ray $\Sigma_{4,3}$, the oriented segment $\Sigma_0=(-1,0)$, and the positively-oriented unit circle $\Sigma^\circ$.
}
\label{fig:PIV-Sigma-Hermite}
\end{figure}
\label{rem:gH-simplification}
\end{remark}
The simplification observed in Remark~\ref{rem:gH-simplification} suggests a connection between the gH rational solutions of Painlev\'e-IV and pseudo-orthogonal polynomials via the Riemann-Hilbert approach of Fokas, Its, and Kitaev \cite{FokasIK:1991}.  Indeed, setting $\mathbf{W}(\lambda;x):=\sigma_1\mathbf{Y}(\lambda;x)\lambda^{-\Theta_0\sigma_3}\sigma_1$ and
%using \eqref{eq:gH-3-parameters} to define $(\Theta_0,\Theta_\infty)$, 
taking the parameters $(\Theta_0,\Theta_\infty)$ to be given by $(\Theta_{0,\mathrm{gH}}^{[3]}(m,n),\Theta_{\infty,\mathrm{gH}}^{[3]}(m,n))$ as defined in Table~\ref{tab:gH}, it is easy to see that $\mathbf{W}(\lambda;x)$ is analytic for $|\lambda|\neq 1$, obeys the jump condition
\eq
\mathbf{W}_+(\lambda;x)=\mathbf{W}_-(\lambda;x)\bpm 1 & w_M(\lambda;x)\\ 0 & 1\epm,\quad |\lambda|=1,\quad
w_M(\lambda;x):=\lambda^{-M}\ee^{-(\lambda^2+2x\lambda)},\quad M:=2\Theta_0,
\label{eq:W-jump}
\endeq
where the jump contour is given counterclockwise orientation, and satisfies the normalization condition 
\eq
\lim_{\lambda\to\infty}\mathbf{W}(\lambda;x)\lambda^{-N\sigma_3}=\mathbb{I},\quad N:=\Theta_\infty+\Theta_0.
\label{eq:W-norm}
\endeq
Noting that $M$ and $N$ are integers, this is the Fokas-Its-Kitaev Riemann-Hilbert problem. If $\mathbf{W}(\lambda;x)$ exists, then $\det(\mathbf{W}(\lambda;x))=1$.  It is well-known and easy to see that solvability requires $N\ge 0$, because otherwise the fact that the first column of $\mathbf{W}(\lambda;x)$ is entire, which follows from the jump condition \eqref{eq:W-jump}, together with 
the normalization condition \eqref{eq:W-norm} implies that the first column of $\mathbf{W}$ vanishes identically by Liouville's Theorem; but this is inconsistent with $\det(\mathbf{W})=1$.  If the solution $\mathbf{W}(\lambda;x)$ exists for $N=0,1,2,\dots$, the matrix element $W_{11}(\lambda;x)$ is the monic pseudo-orthogonal polynomial\footnote{\emph{Pseudo-orthogonality} of monic polynomials $\pi_m(\lambda;x)$ and $\pi_n(\lambda;x)$ of degrees $m$ and $n$ respectively means that for some norming constants $h_n$,
\[
\oint_{|\lambda|=1}\pi_m(\lambda;x)\pi_n(\lambda;x)w_M(\lambda;x)\,\dd\lambda = h_n\delta_{mn},
\]
which is not proper orthogonality because the left-hand side does not define a Hermitian inner product.} of degree $N$ with respect to the weight $w_M(\lambda;x)$ on the unit circle.  However such a polynomial can only exist if $M>0$ because otherwise the weight is analytic for $|\lambda|<1$ and hence every polynomial is pseudo-orthogonal to every monomial by Cauchy's Theorem.  Likewise, the matrix element $W_{21}(\lambda;x)$ is a polynomial in $\lambda$ of degree at most $N-1$ in terms of which the matrix element $W_{22}(\lambda;x)$ is expressed as a Cauchy integral against the weight; the condition \eqref{eq:W-norm} then requires that $W_{22}(\lambda;x)=\lambda^{-N}+\bo(\lambda^{-(N+1)})$ as $\lambda\to\infty$ which leads to a contradiction unless $N<M+1$ (in other words, given $M=1,2,3,\dots$ there can only be finitely many pseudo-orthogonal polynomials of degrees $N=0,1,\dots,M$).
In terms of $(\Theta_0,\Theta_\infty)$, the conditions $M=1,2,3,\dots$ and $N=1,2,\dots, M$ correspond precisely to the points of $\Lambda_\mathrm{gH}^{[3]+}$.  While $\mathbf{W}(\lambda;x)$ and hence $\mathbf{Y}(\lambda;x)$ can also exist for $N=0$ and $M=1,2,3,\dots$, these additional values do not yield a solution $u(x)$ of the Painlev\'e-IV equation via the formula \eqref{eq:u-ucirc} because $\mathbf{W}(\lambda;x)$ is upper-triangular for $N=0$, so $\mathbf{Y}(\lambda;x)$ is lower-triangular, and hence $Y^\infty_{1,12}(x)\equiv 0$.  

In \cite{Buckingham18}, the gH rational solutions of Painlev\'e-IV were studied by means of another system of pseudo-orthogonal polynomials obtained by further developing the method of Bertola and Bothner \cite{BertolaB:2015}.  It is easy to relate the Riemann-Hilbert representation used in \cite{Buckingham18} to the gH Riemann-Hilbert problem for $\mathbf{Y}(\lambda;x)$.  Indeed, from the solution $\mathbf{Y}(\lambda;x)$ of Riemann-Hilbert Problem~\ref{rhp:general} in the gH type-$3$ case, set
%a matrix $\mathbf{Z}^{(j,k)}_0(x)$ via
%\eq
%\mathbf{Z}^{(j,k)}_0(x):=\lim_{\lambda\to 0}\mathbf{Y}^{(j,k)}(\lambda;x)\lambda^{-\Theta_0\sigma_3}.
%\endeq
%Then, we set
%\eq
%\mathbf{M}(\zeta;y):=\left(\frac{\ee^{\ii\pi j/4}}{\sqrt{2\pi}}\right)^{\sigma_3}\ii^{(k+1)\sigma_3}\sigma_1\mathbf{Z}^{(j,k)}_0(\ii y)^{-1}\mathbf{Y}^{(j,k)}(-\ii\zeta^{-1};\ii y)(-\ii\zeta^{-1})^{(1+j)\sigma_3/2}\sigma_1\left(\frac{\ee^{\ii\pi j/4}}{\sqrt{2\pi}}\right)^{-\sigma_3}.
%\endeq
\eq
\mathbf{M}(\zeta;y):=\left(\frac{\ee^{\frac{1}{4}\ii\pi (2\Theta_\infty-1)}}{\sqrt{2\pi}}\right)^{\sigma_3}
\ii^{(\Theta_\infty+\Theta_0)\sigma_3}\sigma_1\mathbf{Y}^0_0(\ii y)^{-1}\mathbf{Y}(-\ii\zeta^{-1};\ii y)(-\ii\zeta^{-1})^{\Theta_\infty\sigma_3}\sigma_1\left(\frac{\ee^{\frac{1}{4}\ii\pi (2\Theta_\infty-1)}}{\sqrt{2\pi}}\right)^{-\sigma_3}.
\label{eq:PIV-Hankel-isomonodromy}
\endeq
It is straightforward to check that 
%if $\mathbf{Y}^{(j,k)}(\lambda;x)$ exists solving Riemann-Hilbert Problem~\ref{rhp:H00}, then 
$\mathbf{M}(\zeta;y)$ solves \cite[Riemann-Hilbert Problem 1]{Buckingham18} with parameters 
%$m=-j$ and $n=k+1$.  
$m=1-2\Theta_\infty$ and $n=\Theta_\infty+\Theta_0$, both integers.  In light of the variable transformation $\lambda=-\ii\zeta^{-1}$, the pseudo-orthogonal polynomials in \cite{Buckingham18} are related to the \emph{reciprocal} pseudo-orthogonal polynomials encoded in the matrix $\mathbf{W}(\lambda;x)$.
The connection between $\mathbf{Y}(\lambda;x)$ and $\mathbf{M}(\zeta;y)$ is analogous to an observation made in \cite{MillerS:2017}, namely that the Riemann-Hilbert problem encoding the Yablonskii-Vorob'ev polynomials found by Bertola and Bothner \cite{BertolaB:2015} is explicitly related to the inverse monodromy problem for the Flaschka-Newell Lax pair for the Painlev\'e-II rational solutions built from those polynomials.

\subsection{Notation}
\label{sec:notation}
We define the three Pauli matrices
\eq
\sigma_1:=\bpm 0&1\\1 & 0\epm,\quad\sigma_2:=\bpm 0 & -\ii\\\ii & 0\epm,\quad\sigma_3:=\bpm 1 & 0\\0 & -1\epm.
\label{eq:Pauli-notation}
\endeq
It will be convenient to have some compact notation for $2\times 2$ matrices having certain structure; thus given a complex number $a$ we define unit determinant lower triangular, upper triangular, diagonal, and ``twist'' matrices by
\eq
\mathbf{L}(a):=\bpm 1&0\\a & 1\epm,\quad \mathbf{U}(a):=\bpm 1 & a\\0 & 1\epm,\quad
\mathbf{D}(a):=a^{\sigma_3}=\bpm a & 0\\ 0 & a^{-1}\epm,\quad \mathbf{T}(a):=\bpm 0 & -a^{-1}\\a & 0\epm.
\label{eq:matrix-factors-notation}
\endeq
In terms of these elementary matrices we will frequently use the following factorizations, which assume $ad-bc=1$:
\eq
\begin{split}
\bpm a & b\\c & d\epm &= 
\mathbf{L}(ca^{-1})\mathbf{D}(a)\mathbf{U}(ba^{-1}),\quad a\neq 0,\quad\text{(``LDU'')},\\
&=\mathbf{L}(db^{-1})\mathbf{T}(-b^{-1})\mathbf{L}(ab^{-1}),\quad b\neq 0,\quad\text{(``LTL'')},\\
&=\mathbf{U}(bd^{-1})\mathbf{D}(d^{-1})\mathbf{L}(cd^{-1}),\quad d\neq 0,\quad\text{(``UDL'')},\\
&=\mathbf{U}(ac^{-1})\mathbf{T}(c)\mathbf{U}(dc^{-1}),\quad c\neq 0,\quad\text{(``UTU'')}.\\
%&=\bpm 1/D & B\\0 & D\epm\bpm 1 & 0\\C/D & 1\epm\\
%&=\bpm 1 & 0\\D/B & 1\epm\bpm 0 & B\\-1/B & 0\epm\bpm 1 & 0\\A/B & 1\epm\\
%&=\bpm A & 0\\ C & 1/A\epm\bpm 1 & B/A\\0 & 1\epm\\
%&=\bpm 1 & A/C\\0 & 1\epm\bpm 0 & -1/C\\ C & 0\epm\bpm 1 & D/C\\0 & 1\epm.
\end{split}
\label{eq:general-factorizations}
\endeq

For a function $f$ defined on the complement of an oriented arc in the complex plane, we use subscripts ``$+$'' (resp., ``$-$'') to denote the boundary value taken at a given point on the arc from the left (resp., right):  $f_\pm$.  We sometimes abbreviate the average and difference of these boundary values by writing 
\eq
\langle f\rangle:=\tfrac{1}{2}(f_++f_-)\quad\text{and}\quad \Delta f:=f_+-f_-.
\endeq
Throughout this paper, for a quantity $Q$ we use a ``dot'' notation $\dot{Q}$ to indicate a corresponding approximation.
Finally, for expressions applicable to either family, gH or gO, we frequently use a generic subscript $\mathrm{F}$ (for ``family'').

\subsection*{Acknowledgements}
\input{Acknowledgements.tex}

%% file: BriefIntro.tex
\subsection{Overview}

The Painlev\'e-IV equation 
\eq
u''=\frac{(u')^2}{2u}+\frac{3}{2}u^3+4xu^2+2(x^2+1-2\Theta_\infty)u-\frac{8\Theta_0^2}{u}, \quad u:\mathbb{C}\to\mathbb{C} \text{ with parameters } \Theta_0,\Theta_\infty\in\mathbb{C},
\label{p4}
\endeq
for a function $u=u(x)$ is a fundamental equation of mathematical physics with 
applications ranging from nonlinear wave equations \cite{BassomCH:1996} and 
quantum gravity \cite{FokasIK:1991} to orthogonal polynomials \cite{DaiK:2009} 
and random matrix theory \cite{ChenF:2006,ForresterW:2001,OsipovSZ:2010}.  
Equation \eqref{p4} is well known to have a (unique) rational solution if 
$\Theta_0$ and $\Theta_\infty$ belong to certain real discrete sets described 
precisely in Section \ref{sec:rational-PIV-solutions}.  These rational 
solutions have attracted attention both for their wide number of applications 
(see Section \ref{sec:rational-PIV-applications} for references) and for the 
intriguing structure of their zeros and poles.  There are two distinct 
\emph{families} of rational solutions:  those that can be expressed via 
generalized Hermite polynomials (the \emph{gH family}), and those that can be 
expressed via generalized Okamoto polynomials (the \emph{gO family}).  In the 
complex $x$-plane, the zeros and poles of a gH solution appear to form a 
quasi-rectangular grid (see Figure \ref{fig:theta-map-gH}).  
The aspect ratio of the quasi-rectangle depends on the angle in the real 
$(\Theta_0,\Theta_\infty)$-plane, while the number of zeros and poles grows 
with $|\Theta_0|$ and $|\Theta_\infty|$.  The gH solutions are naturally 
divided into \emph{types} 1, 2, and 3 depending on the angle in the 
$(\Theta_0,\Theta_\infty)$-plane.  The microstructure of zeros and poles is 
different for each type.  We label the gH solutions 
$u_\mathrm{gH}^{[j]}(x;m,n)$, where $m$ and $n$ are non-negative integers.  
The zeros and poles of the gO solutions are more complicated, appearing to 
form a quasi-rectangle with quasi-triangles attached to each edge (see 
Figure \ref{fig:theta-map-gO}).  We divide the gO solutions into types 1, 2, 
and 3, and label solutions as $u_\mathrm{gO}^{[j]}(x;m,n)$.  There is an 
important difference from the gH family:  the gO solutions of each type fill 
\emph{two} sectors in the $(\Theta_0,\Theta_\infty)$-plane.  Using our 
conventions, the integers $m$ and $n$ will both be nonnegative in one of 
these sectors and both nonpositive in the other.  

For either family, if one fixes an angle in the 
$(\Theta_0,\Theta_\infty)$-plane and writes $T:=|\Theta_0|$ and 
$y:=T^{-\frac{1}{2}}x$, the region in which zeros and poles lie becomes more 
clearly defined in the $y$-plane as $T\to\infty$.  In this work  we 
analytically determine the boundaries of the quasi-rectangles and (for the gO 
family) the quasi-triangles in the $y$-plane.  We prove that in the exterior 
of the rectangular and triangular regions, the scaled rational Painlev\'e-IV 
functions $T^{-\frac{1}{2}}u_\mathrm{F}^{[j]}(T^\frac{1}{2}y;m,n)$ are 
asymptotically approximated by algebraic equilibrium solutions of the 
autonomous approximating equation 
\eq
\dot{U}_{\zeta\zeta}=\frac{(\dot{U}_\zeta)^2}{2\dot{U}}+\frac{3}{2}\dot{U}^3+4y\dot{U}^2+(2y^2+4\kappa)\dot{U}-\frac{8}{\dot{U}}
\label{eq:Approximating-ODE-Second-Order1}
\endeq
(see Theorems \ref{thm:HermiteExterior} and \ref{thm:OkamotoExterior}).
Furthermore, for each $y$ in the rectangular and triangular domains, we show 
that, as a function of $\zeta$, 
$T^{-\frac{1}{2}}u_\mathrm{F}^{[j]}(T^\frac{1}{2}y+T^{-\frac{1}{2}}\zeta;m,n)$
is asymptotically approximated by a classical elliptic function solution of 
\eqref{eq:Approximating-ODE-Second-Order1} (see Theorems 
\ref{thm:Hermite-elliptic} and \ref{thm:Okamoto-elliptic}).  All of these 
results are new, with the exception of the exterior asymptotics for the gH 
family, which were obtained in \cite{Buckingham18}.  Our results assume that 
$y$ does not lie on the boundaries of the rectangular or triangular regions.  
We furthermore exclude the angles $\pm\tfrac{1}{4}\pi$, $\pm\pi$, and 
$\pm\tfrac{3}{4}\pi$ in the $(\Theta_0,\Theta_\infty)$-plane, i.e. those 
angles for which the zero/pole region collapses to a line segment (for gH 
solutions) or two triangles (for gO solutions). 

Our analysis begins with the derivation of two Riemann-Hilbert problems 
encoding the gO and gH solutions, respectively (see Section~\ref{sec:intro-Methodology} and 
Appendix \ref{app:Isomonodromy}).  The gO Riemann-Hilbert problem has not 
appeared in the literature previously\footnote{The Stokes data, but not the 
connection matrices, were computed in \cite{NovokshenovS14}.  See Remark 
\ref{rem:NovokshenovS14} in 
Appendix \ref{sec:RHPgO} for more details.}, while we show that the gH 
Riemann-Hilbert problem can be transformed to the problem in 
\cite{Buckingham18}.  In order to cut down on the number of cases studied, we 
use the Boiti-Pempinelli symmetry \eqref{eq:rotation-symmetry} to 
immediately write type-2 solutions (of either family) in terms of type-1 
solutions.  We furthermore use the Lukashevich-Gromak B\"acklund 
transformation \eqref{eq:Baecklund-3-to-1} to express type-1 solutions in 
terms of type-3 solutions, so it is only necessary to carry out the 
Riemann-Hilbert analysis for type-3 parameters.  However, this 
simplification comes at a price: our analysis of type-1 solutions is carried 
out using different (i.e. type-3) parameters.  For each type-1 result we 
carry out the additional analysis necessary to rewrite the formulas in terms 
of the natural type-1 parameters.  

In the remainder of the introduction we introduce the rational Painlev\'e-IV 
solutions, their applications, and the basic B\"acklund transformation we 
will use (Section \ref{sec:rational-PIV-solutions}), formally derive the 
autonomous approximating equation (Section \ref{sec:scaling}), state our 
results (Section~\ref{sec:Results}), and introduce the Riemann-Hilbert problems 
(Section \ref{sec:intro-Methodology}).  In Section~\ref{sec:Basic-Principles}, we perform the 
preliminary steps of the Riemann-Hilbert analysis that are common to all the 
cases we study.  In Sections \ref{sec:Exterior} and \ref{sec:Exterior-gH}, 
we prove the exterior asymptotics for the gO and GH families, respectively.  
In Section \ref{sec:OutsideDomain}, we determine the boundaries of the 
different asymptotic domains in the $y$-plane.  In Sections \ref{sec:Boutroux} 
and \ref{sec:G1}, we prove the asymptotic behavior in the regions where the 
Painlev\'e-IV solutions are described by elliptic functions (see also 
Appendices \ref{app:DiagramsAndTables} and \ref{app:Effective} for more 
details).  Appendix 
\ref{app:Isomonodromy} contains the derivations of the Riemann-Hilbert 
problems we study.  In Appendix \ref{app:Equilateral}, we prove that the corner points 
of the quasi-triangular regions define an equilateral triangle.  
Finally, Appendix~\ref{app:Alternate} introduces an alternate method 
for analyzing type-1 solutions.

%% file: Acknowledgements.tex
The authors thank Davide Masoero and Pieter Roffelsen for useful discussions and Guilherme Silva for information about trajectories of rational quadratic differentials and for suggesting the possibility of representing the arcs of the curves $\partial\HermiteExterior(\kappa)$ and $\partial\OkamotoExterior(\kappa)$ in the $y_0$-plane in terms of such trajectories.  R. J. Buckingham was supported by National Science Foundation (grant DMS-1615718).  P. D. Miller was supported by the National Science Foundation (grants DMS-1513054 and DMS-1812625).

%% file: OkamotoAsymp.tex
\section{Asymptotic analysis of $\mathbf{Y}(\lambda;x)$ for $(\Theta_0,\Theta_\infty)$ large:  Basic principles}
\label{sec:Basic-Principles}
%Recalling from Section~\ref{sec:scaling} the scalings \eqref{eq:Thetas-scaling}, 
In light of the explicit and trivial relation \eqref{eq:symmetry-1-2} between $u^{[2]}_\mathrm{F}(x;m,n)$ and $u^{[1]}_\mathrm{F}(x;m,n)$ for both families $\mathrm{F}=\mathrm{gH}$ and $\mathrm{F}=\mathrm{gO}$, to prove our results it will be sufficient to 
consider only rational solutions of types $1$ and $3$.  Moreover, since \eqref{eq:symmetry-1-from-3} and \eqref{eq:u-ucirc} together imply that the rational solutions of both types $1$ and $3$ are simultaneously encoded in Riemann-Hilbert Problem~\ref{rhp:general}, it is only necessary to study the latter problem in the situation that the parameters $(\Theta_0,\Theta_\infty)$ correspond to a rational solution of type $3$ in either family.  

We therefore assume that the parameters $(\Theta_0,\Theta_\infty)$ are large in either $\Lambda_\mathrm{gO}\cap (W^{[3]+}\cup W^{[3]-})$ or $\Lambda_\mathrm{gH}^{[3]+}$ for the gO and gH families, respectively.  This implies that, in terms of the parameters $T,s,\kappa$ from \eqref{eq:Thetas-scaling}, $T>0$ will be the large parameter and $\kappa\in (-1,1)$.  For the gO family we allow both signs $s=\pm 1$ for $\Theta_0$ to access both sectors $W^{[3]\pm}$, while for the gH family it is enough to consider only $s=1$ since $\Lambda_\mathrm{gH}^{[3]+}\subset W^{[3]+}$.

\subsection{Scaling of Riemann-Hilbert Problem~\ref{rhp:general}}
In light of the parametrization \eqref{eq:Thetas-scaling} of $(\Theta_0,\Theta_\infty)$ by $T,s,\kappa$, it is convenient to introduce the following scalings into the solution $\mathbf{Y}(\lambda;x)$ of Riemann-Hilbert Problem~\ref{rhp:general}:
\eq
x=T^{\frac{1}{2}}y,\quad y=y_0+T^{-\frac{1}{2}}\zeta,\quad\text{and}\quad \lambda=\frac{1}{2}T^{\frac{1}{2}}z.
\label{eq:x-y-lambda-z}
\endeq
We set $\mathbf{M}^{(T,s,\kappa)}(z;y):=(\tfrac{1}{2}T^{\frac{1}{2}})^{-\kappa T\sigma_3}\mathbf{Y}(\lambda;x)$, and we write $\mathbf{M}(z)=\mathbf{M}^{(T,s,\kappa)}(z;y)$ when we wish to suppress the dependence on parameters.
Then under the scalings \eqref{eq:x-y-lambda-z} the exponent in the jump conditions of Riemann-Hilbert Problem~\ref{rhp:general} becomes
\eq
\lambda^2+2x\lambda=2T\phi(z;y),\quad\phi(z;y):=\frac{1}{8}z^2+\frac{1}{2}yz,
\label{eq:phi-exponent}
\endeq
and $\mathbf{M}(z)z^{-sT\sigma_3}$ is bounded as $z\to 0$ while $\mathbf{M}(z)z^{-\kappa T\sigma_3}\to\mathbb{I}$ as $z\to\infty$.  Because the jump matrices on all arcs of $\Sigma$ are all entire functions of $\lambda$ and are cyclically consistent at all self-intersection points in $\mathbb{C}\setminus\{0\}$ due to the consistency relations \eqref{eq:consistency}, by elementary substitutions in the four sectors between the circles of radius $|z|=\tfrac{1}{2}T^{\frac{1}{2}}$ and $|z|=1$ we may simply take the jump contour for $\mathbf{M}(z)$ to again be the original unscaled jump contour $\Sigma$, now in the $z$-plane.  Since the constant pre-factor $(\tfrac{1}{2}T^{\frac{1}{2}})^{-\kappa T\sigma_3}$ does not affect any jump conditions, the jump matrices for $\mathbf{M}(z)$ are precisely the same as those of $\mathbf{Y}(\lambda;x)$ on the same arcs of $\Sigma$ except that the exponents are replaced in each case according to \eqref{eq:phi-exponent}, and $\Theta_0$ and $\Theta_\infty$ are replaced with $sT$ and $-\kappa T$ respectively.  
%By the assumptions on $\{(m_k,n_k)\}_{k=1}^\infty$ mentioned above, we will assume that for every $\epsilon>0$, $\kappa>-\epsilon$ holds for $T$ large enough, and that $\kappa$ is bounded away from $1$.

\begin{remark}
Given a family $\mathrm{F}=\mathrm{gO}$ or $\mathrm{F}=\mathrm{gH}$, the parameters $T$, $s$, $\kappa$, and the auxiliary variable $y$ appearing in $\mathbf{M}(z)=\mathbf{M}^{(T,s,\kappa)}(z;y)$ are naturally related to the function $u(x)=u_\mathrm{F}^{[3]}(x;m,n)$ solving the Painlev\'e-IV equation \eqref{p4} for parameters $(\Theta_0,\Theta_\infty)=(\Theta_{0,\mathrm{F}}^{[3]}(m,n),\Theta_{\infty,\mathrm{F}}^{[3]}(m,n))$.  To study the function $u_\mathrm{gH}^{[3]}(x;m,n)$, we therefore relate these quantities to the integer parameters $(m,n)\in\mathbb{Z}_{\ge 0}^2$ by
\eq
T:=\tfrac{1}{2}+\tfrac{1}{2}(m+n),\quad s:=+1,\quad \kappa:=-\frac{1+n-m}{1+m+n},\quad x=\sqrt{\frac{1+m+n}{2}}y_0+\sqrt{\frac{2}{1+m+n}}\zeta,
\label{eq:gH-type3-RHP-parameters}
\endeq
and observe that as $m,n\to+\infty$ with $n=\rho m$ for a fixed aspect ratio $\rho>0$, $T\to\infty$ and $\kappa\to (1-\rho)/(1+\rho)\in (-1,1)$.  Likewise, to study the function $u_\mathrm{gO}^{[3]}(x;m,n)$, for integer parameters $(m,n)\in\mathbb{Z}_{>0}^2$ with $mn>0$ we use instead
\eq
T:=|\tfrac{1}{6}+\tfrac{1}{2}(m+n)|,\quad s:=\mathrm{sgn}(m+n),\quad\kappa:=-\frac{3+3n-3m}{|1+3m+3n|},\quad x=\sqrt{\frac{|1+3m+3n|}{6}}y_0+\sqrt{\frac{6}{|1+3m+3n|}}\zeta,
\label{eq:gO-type3-RHP-parameters}
\endeq
and observe that as $m,n\to\infty$ with $n=\rho m$ for a fixed aspect ratio $\rho>0$, $T\to\infty$ and $\kappa\to \mathrm{sgn}(m+n)(1-\rho)/(1+\rho)\in (-1,1)$.

However, the function $u_\mathrm{F}^{[1]}(x;m,n)$ satisfies \eqref{p4} for different parameters, namely for $(\Theta_{0,\tw},\Theta_{\infty,\tw})=(\Theta_{0,\mathrm{F}}^{[1]}(m,n),\Theta_{\infty,\mathrm{F}}^{[1]}(m,n))$ related to $(\Theta_0,\Theta_\infty)$ via the symmetry $\mathcal{S}_\tw$ defined in \eqref{eq:Baecklund-3-to-1}.  Writing $(\Theta_0,\Theta_\infty)$ in terms of the integer indices $(m,n)$ for this function therefore requires inverting the mapping $\mathcal{S}_\tw$ on the parameters as follows.  For the function $u_\tw(x)=u_\mathrm{gH}^{[1]}(x;m,n)$, the parameters in Riemann-Hilbert Problem~\ref{rhp:general} become
\eq
(\Theta_0,\Theta_\infty)=\mathcal{S}_\tw^{-1}\circ (\Theta_{0,\mathrm{gH}}^{[1]}(m,n),\Theta_{\infty,\mathrm{gH}}^{[1]}(m,n)) = (\tfrac{1}{2}m+\tfrac{1}{2}n,-\tfrac{1}{2}m+\tfrac{1}{2}n),
\label{eq:gH-index-twist}
\endeq
yielding for $\mathbf{M}(z)$ the parameters
\eq
T:=\tfrac{1}{2}(m+n),\quad s=+1,\quad \kappa:=\frac{m-n}{m+n},
\label{eq:gH-type1-RHP-parameters}
\endeq
in which $T\to +\infty$ and $\kappa= (1-\rho)/(1+\rho)\in (-1,1)$ as $m,n\to+\infty$ with $n=\rho m$,
while for the function $u_\tw(x)=u_\mathrm{gO}^{[1]}(x;m,n)$, we have instead
\eq
(\Theta_0,\Theta_\infty)=\mathcal{S}_\tw^{-1}\circ (\Theta_{0,\mathrm{gO}}^{[1]}(m,n),\Theta_{\infty,\mathrm{gO}}^{[1]}(m,n)) = (-\tfrac{1}{3}+\tfrac{1}{2}m+\tfrac{1}{2}n,-\tfrac{1}{2}m+\tfrac{1}{2}n),
\label{eq:gO-index-twist}
\endeq
yielding
\eq
T:=\tfrac{1}{2}|m+n-\tfrac{2}{3}|,\quad s=\mathrm{sgn}(m+n),\quad\kappa:=\frac{m-n}{|m+n-\tfrac{2}{3}|}
\label{eq:gO-type1-RHP-parameters}
\endeq
in which $T\to+\infty$ and $\kappa\to \mathrm{sgn}(m+n)(1-\rho)/(1+\rho)$ as $m,n\to\infty$ with $mn>0$ and $n=\rho m$.
We emphasize that in this case, $(\Theta_0,\Theta_\infty)$ are \emph{not} the parameters in \eqref{p4} for which the indicated type-$1$ function satisfies the Painlev\'e-IV equation, but they are the parameters in Riemann-Hilbert Problem~\ref{rhp:general} for which this function is encoded as $u_\tw(x)$ given by \eqref{eq:u-ucirc}.  

The analysis we present in the rest of this section and in Sections~\ref{sec:Exterior}--\ref{sec:G1} will refer to the quantities $T$, $s$, and $\kappa$ defined as above depending on which family and type of rational function is being considered.  A remaining issue in interpreting the results of a large-$T$ asymptotic analysis of $\mathbf{M}(z)$ is that while it is natural in light of the scalings \eqref{eq:Thetas-scaling} to write $x$ in the form $x=T^{\frac{1}{2}}y_0+T^{-\frac{1}{2}}\zeta$ as indicated above for the type-$3$ rational solutions, for the type-$1$ rational functions in the family F we need to use $|\Theta^{[1]}_{0,\mathrm{F}}(m,n)|$ in place of $|\Theta_0|$ in defining $y_0$ and $\zeta$.  This amounts to replacing 
\eq
y_0\;\text{with}\; \sqrt{\frac{|\Theta^{[1]}_{0,\mathrm{F}}(m,n)|}{|\Theta_0|}}y_0\quad\text{and}\quad
\zeta\;\text{with}\;\sqrt{\frac{|\Theta_0|}{|\Theta_{0,\mathrm{F}}^{[1]}(m,n)|}}\zeta
\label{eq:y0zeta-type1-substitutions}
\endeq
in all final formul\ae.  Note that using Tables~\ref{tab:gH} and \ref{tab:gO}, and taking $\Theta_0$ to be given by \eqref{eq:gH-index-twist} or \eqref{eq:gO-index-twist} respectively,
\eq
\frac{|\Theta_{0,\mathrm{gH}}^{[1]}(m,n)|}{|\Theta_0|} = \frac{n}{m+n}\quad\text{and}\quad
\frac{|\Theta_{0,\mathrm{gO}}^{[1]}(m,n)|}{|\Theta_0|} = \frac{n-\tfrac{1}{3}}{m+n-\tfrac{2}{3}}.
\label{eq:y0zeta-type1-substitutions-factors}
\endeq
Equivalently, since $\Theta_{0,\mathrm{F}}^{[1]}(m,n)=\Theta_{0,\tw}$, in terms of the parameters $s$ and $\kappa$ related to the indices $(m,n)$ for the $\mathrm{F}=\mathrm{gH}$ and $\mathrm{F}=\mathrm{gO}$ families by \eqref{eq:gH-type1-RHP-parameters} and \eqref{eq:gO-type1-RHP-parameters} respectively,
\eq
\frac{|\Theta_{0,\mathrm{F}}^{[1]}(m,n)|}{|\Theta_0|}=\frac{1}{2}(1-s\kappa).
\label{eq:theta0-ratio}
\endeq 
\label{rem:other-parameters}
\end{remark}

\subsection{Trivially equivalent Riemann-Hilbert problems for $\mathbf{M}(z)$}
\label{sec:trivially-equivalent-RHPs}
\subsubsection{The gO case}
We further observe that, by a similar argument using analyticity of jump matrices and cyclic consistency at nonzero self-intersection points, the jump contour $\Sigma$ can be replaced by 
a qualitatively similar jump contour consisting of 
\begin{itemize}
\item
an arbitrary Jordan curve $C$ enclosing the origin and divided into arcs $\Sigma_j$, $j=1,2,3,4$ (the indicated sub-arcs are homeomorphic in $\mathbb{C}\setminus\{0\}$ with the corresponding curves on the unit circle shown in Figure~\ref{fig:PIV-Sigma}), 
\item
an arbitrary simple arc $\Sigma_0$ in the interior of $C$ that connects the junction point of $\Sigma_3$ and $\Sigma_4$ to the origin, and
\item
four arbitrary disjoint simple arcs $\Sigma_{j,k}$, unbounded in one direction and connecting $z=\infty$ with the junction point of $\Sigma_j$ and $\Sigma_k$ such that the approach to $z=\infty$ is in the (vertical or horizontal) direction shown in Figure~\ref{fig:PIV-Sigma}.
\end{itemize}
In general, the union of $\Sigma_0$ and $\Sigma_{4,3}$ should be taken as the branch cut for the functions $z^{-sT\sigma_3}$ and $z^{-\kappa T\sigma_3}$, and the branches of these functions remain principal for sufficiently large $|z|$.  The formula for the jump matrix on each arc of $\Sigma$ after the deformation is exactly the same in each case as before the deformation.  

To study $\mathbf{M}(z)$ in the case that the monodromy data corresponds to the family of gO rational solutions of Painlev\'e-IV (see \eqref{eq:gO-Stokes}--\eqref{eq:gO-connection}), it will be useful to introduce two modifications of the Riemann-Hilbert conditions for $\mathbf{M}(z)$ that do not preserve the topology of $\Sigma$.
\begin{figure}[h]
\begin{center}
\includegraphics{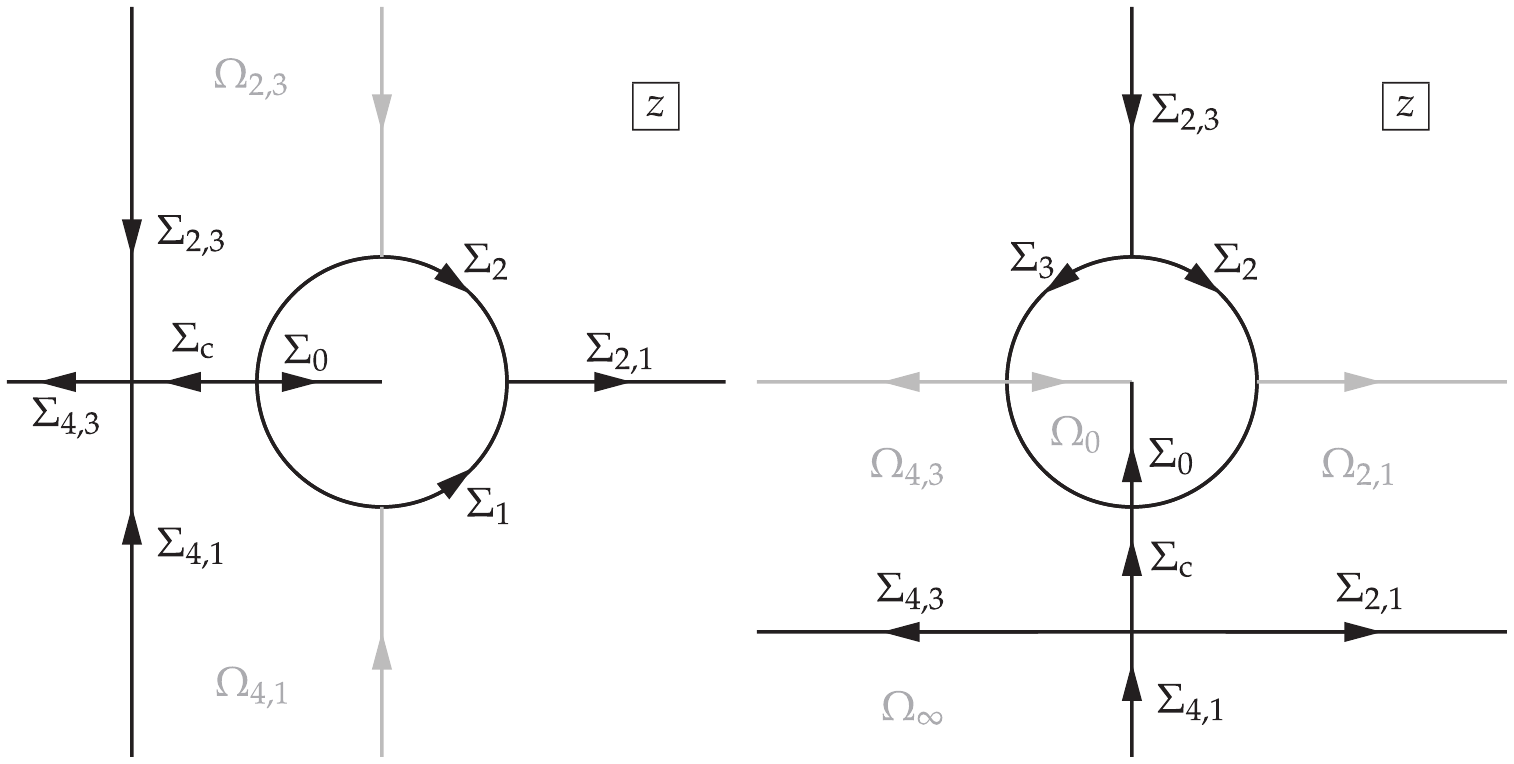}
\end{center}
\caption{Left:  the jump contour for the ``leftward'' deformation of $\mathbf{M}(z)$.  Right:  the jump contour for the ``downward'' deformation of $\mathbf{M}(z)$.}
\label{fig:PIV-Sigma-Deform}
\end{figure}
Beginning with the jump contour shown in Figure~\ref{fig:PIV-Sigma} and referring to the left-hand panel of Figure~\ref{fig:PIV-Sigma-Deform} we define a ``leftward'' deformation of the jump contour for $\mathbf{M}(z)$ by making the following piecewise analytic substitution in the domains $\Omega_{2,3}$ and $\Omega_{4,1}$ exterior to the unit circle in the $z$-plane (recall the notation \eqref{eq:matrix-factors-notation}):
\eq
\mathbf{M}(z)\mapsto \mathbf{M}(z)
%\bpm
%1 & -\tfrac{1}{2}\ii\ee^{2T\phi(z;y)}\\0 & 1\epm,
\mathbf{U}(-\tfrac{1}{2}\ii \ee^{2T\phi(z;y)}),
\quad z\in\Omega_{2,3},
\endeq
\eq
\mathbf{M}(z)\mapsto \mathbf{M}(z)
%\bpm
%1 & \tfrac{1}{2}\ii\ee^{2T\phi(z;y)}\\0 & 1\epm,
\mathbf{U}(\tfrac{1}{2}\ii\ee^{2T\phi(z;y)}),\quad z\in\Omega_{4,1},
\endeq
and elsewhere we leave $\mathbf{M}(z)$ unchanged.  This results in the same jump conditions as indicated in Riemann-Hilbert Problem~\ref{rhp:general} with monodromy data taken from \eqref{eq:gO-Stokes}--\eqref{eq:gO-connection} on the corresponding labeled arcs (but in the $z$-plane, and with the exponents modified as indicated in \eqref{eq:phi-exponent}), and a new jump condition on the arc labeled $\Sigma_\mathrm{c}$ in the left-hand panel of Figure~\ref{fig:PIV-Sigma-Deform}, namely
\eq
\mathbf{M}_+(z)=\mathbf{M}_-(z)\ee^{-2\pi\ii T\kappa}
%\bpm 0 & \tfrac{1}{2}\ii\ee^{2T\phi(z;y)}\\2\ii\ee^{-2T\phi(z;y)} & 0\epm,
\mathbf{T}(2\ii\ee^{-2T\phi(z;y)}),
\quad
z\in\Sigma_\mathrm{c}.
\endeq
Likewise, referring to the right-hand panel of Figure~\ref{fig:PIV-Sigma-Deform} we can define a ``downward'' deformation of the jump contour for $\mathbf{M}(z)$ by making the following analytic substitutions in the indicated domains:
\eq
\mathbf{M}(z)\mapsto\mathbf{M}(z)\ee^{2\pi\ii sT\sigma_3},\quad z\in\Omega_0,
\endeq
\eq
\mathbf{M}(z)\mapsto\mathbf{M}(z)\ee^{2\pi\ii T\kappa}
%\bpm 1 & 0\\-2\ii\ee^{-2T\phi(z;y)} & 1\epm,
\mathbf{L}(-2\ii\ee^{-2T\phi(z;y)}),
\quad z\in\Omega_{4,3},
\endeq
\eq
\mathbf{M}(z)\mapsto\mathbf{M}(z)
%\bpm 1 & 0\\ 2\ii\ee^{-2T\phi(z;y)} & 1\epm,
\mathbf{L}(2\ii\ee^{-2T\phi(z;y)}),
\quad z\in\Omega_{2,1},
\endeq
and
\eq
\mathbf{M}(z)\mapsto\mathbf{M}(z)
\ee^{2\pi\ii T\kappa},\quad z\in\Omega_\infty.
\label{eq:scalar-change}
\endeq
Once again, the resulting jump conditions on the arcs labeled as in Figure~\ref{fig:PIV-Sigma} correspond to those in Riemann-Hilbert Problem~\ref{rhp:general} with monodromy data given in \eqref{eq:gO-Stokes}--\eqref{eq:gO-connection}, except on the arcs $\Sigma_{4,3}$ and $\Sigma_{4,1}$ where we have instead
\eq
\mathbf{M}_+(z)=\mathbf{M}_-(z)
%\bpm 1 & 0\\2\ii\ee^{-2T\phi(z;y)} & 1\epm,
\mathbf{L}(2\ii\ee^{-2T\phi(z;y)}),
\quad z\in\Sigma_{4,3},
\endeq
\eq
\mathbf{M}_+(z)=\mathbf{M}_-(z)\ee^{2\pi\ii T\kappa}
%\bpm 1 & -\tfrac{1}{2}\ii\ee^{2T\phi(z;y)}\\0 & 1\epm,
\mathbf{U}(-\tfrac{1}{2}\ii\ee^{2T\phi(z;y)}),
\quad z\in\Sigma_{4,1},
\endeq
(as the substitution \eqref{eq:scalar-change} effectively moves the scalar factor $\ee^{-2\pi\ii T\kappa}$ from the jump across $\Sigma_{4,3}$ to that across $\Sigma_{4,1}$) and there is a new jump condition across the arc labeled $\Sigma_\mathrm{c}$ in the right-hand panel of Figure~\ref{fig:PIV-Sigma-Deform} that reads
\eq
\mathbf{M}_+(z)=\mathbf{M}_-(z)\ee^{2\pi\ii T\kappa}
%\bpm 0 & -\tfrac{1}{2}\ii\ee^{2T\phi(z;y)}\\ -2\ii\ee^{-2T\phi(z;y)} & 0\epm,
\mathbf{T}(-2\ii\ee^{-2T\phi(z;y)}),
\quad z\in\Sigma_\mathrm{c}.
\endeq
In interpreting the conditions on $\mathbf{M}(z)$ as $z$ tends to $0$ and $\infty$, one should now replace the principal branch power functions $z^{-sT\sigma_3}$ and $z^{-T\kappa\sigma_3}$ by branches with $\arg(z)\in (-\tfrac{1}{2}\pi,\tfrac{3}{2}\pi)$.  

In the case of both deformations, one can subsequently replace the unit circle in the $z$-plane by any Jordan curve enclosing the origin, and employ similar contour deformations that respect the topology and direction of approach to $z=\infty$.  In the case of the leftward deformation, the union of the deformations of $\Sigma_0$, $\Sigma_\mathrm{c}$, and $\Sigma_{4,3}$ should be taken as the branch cut of $z^{-sT\sigma_3}$ and $z^{-T\kappa\sigma_3}$, which should then be interpreted via $\arg(z)\in (-\pi,\pi)$ near $z=\infty$.  In the case of the downward deformation, the union of the deformations of $\Sigma_0$, $\Sigma_\mathrm{c}$, and $\Sigma_{4,1}$ form the branch cut of these functions, which are to be interpreted via $\arg(z)\in (-\tfrac{1}{2}\pi,\tfrac{3}{2}\pi)$ near $z=\infty$.

Regardless of whether we use the original contour topology or the ``leftward'' or ``downward'' modifications, we will always refer to the jump contour for $\mathbf{M}(z)=\mathbf{M}^{(T,s,\kappa)}(z;y)$ as $\Sigma$.  We will take full advantage of the freedom of choice of $\Sigma$ consistent with the above discussion, allowing $\Sigma$ to depend on $s$, $\kappa$, and $y_0$.  \emph{Importantly however, we will insist that $\Sigma$ is independent of the large parameter $T\gg 1$.}  Nor will it depend on $\zeta$.

\subsubsection{The gH case}
When we consider the matrix $\mathbf{M}(z)$ connected to the solution of Riemann-Hilbert Problem~\ref{rhp:general} with gH monodromy data given in \eqref{eq:gH-Stokes}--\eqref{eq:gH-connection}, there is a corresponding dramatic simplification of the rescaled jump contour as described in Remark~\ref{rem:gH-simplification}.  Indeed, the rescaled version of $\Sigma$ can be taken to appear exactly as shown in Figure~\ref{fig:PIV-Sigma-Hermite}, but now in the $z$-plane.  Here there will be no need for deformations that change the topology of the jump contour, but as in the gO case we may always replace the unit circle in the $z$-plane with any Jordan curve $C$ enclosing the origin, and we may replace the contour arcs lying on the negative real line with arcs still denoted $\Sigma_0$ and $\Sigma_{4,3}$ making up any simple curve that connects $z=0$ with $z=-\infty$ and that intersects $C$ only at one point.  The branch cuts of $z^{-sT\sigma_3}$ and $z^{-\kappa T\sigma_3}$ are then taken to coincide with the latter curve, and the branches are chosen to be principal for large $|z|$.  Also as in the gO case, any modification of $\Sigma$ will be assumed to be independent of $T$ and $\zeta$.

\subsection{Spectral curve and $g$-function}
\label{sec:SpectralCurve}
The analysis in this section applies equally to the matrix $\mathbf{M}(z)$ in all configurations of $\Sigma$, regardless of whether Riemann-Hilbert Problem~\ref{rhp:general} describes gO or gH rational solutions of Painlev\'e-IV.  Indeed, the fact that one can use the same theory of spectral curves to study both families of rational solutions is one of the main advantages of putting both families on the same footing via Riemann-Hilbert Problem~\ref{rhp:general}.

Let $g:\mathbb{C}\setminus\Sigma\to\mathbb{C}$ be analytic with continuous boundary values such that $g'(z)$ is also such a function, and moreover, for some constants $g_0$ and $g_\infty$, 
\eq
g(z)=\begin{cases} -s\log(z)+g_0+\bo(z),&\quad z\to 0,\\
-\kappa\log(z)+g_\infty+\bo(z^{-1}),&\quad z\to\infty
\end{cases}
\quad\implies\quad g'(z)=\begin{cases}-sz^{-1}+\bo(1),&\quad z\to 0,\\
-\kappa z^{-1}+\bo(z^{-2}),&\quad z\to\infty.
\end{cases}
\label{eq:gprime-asymp}
\endeq
Here the branch of $\log(z)$ corresponds to the definition of the power functions $z^{-sT\sigma_3}$ and $z^{-T\kappa\sigma_3}$ as indicated in Section~\ref{sec:trivially-equivalent-RHPs}.
Given such a function, from $\mathbf{M}^{(T,s,\kappa)}(z;y)$ we define a new unknown $\mathbf{N}(z)=\mathbf{N}^{(T,s,\kappa)}(z;y)$ by the substitution
\eq
\mathbf{N}(z):=\ee^{-Tg_\infty\sigma_3}\mathbf{M}(z)\ee^{Tg(z)\sigma_3},\quad z\in\mathbb{C}\setminus\Sigma.
\label{eq:NfromM}
\endeq
Note that from \eqref{eq:gprime-asymp}, $\mathbf{N}(z)\to\mathbb{I}$ as $z\to\infty$.
The induced jump conditions for $\mathbf{N}(z)$ will involve exponentials on the diagonal elements with exponents $\pm T\Delta g(z)$ and on the off-diagonal elements with exponents $\pm 2T(\langle g\rangle(z)-\phi(z;y))$ where $\Delta g(z)$ and $\langle g\rangle(z)$ are defined in terms of the boundary values $g_\pm(z)$ taken on an arc of $\Sigma$ as in Section~\ref{sec:notation}.   Assuming that $\Sigma$ is partitioned into arcs in which either $\Delta g(z)= g_+(z)-g_-(z)$ or $2(\langle g\rangle(z)-\phi(z;y))=g_+(z)+g_-(z)-2\phi(z;y)$ is independent of $z$, we easily see that the function $(g'(z)-\phi'(z;y))^2$ has no jump across any of the arcs of $\Sigma$ and hence is a function analytic for $z\in\mathbb{C}\setminus\{0\}$.  To identify this function, we use \eqref{eq:phi-exponent} and \eqref{eq:gprime-asymp} to examine its behavior near $z=0$:
\eq
(g'(z)-\phi'(z;y))^2=(-sz^{-1}+\bo(1))^2=z^{-2}+\bo(z^{-1}),\quad z\to 0,
\label{eq:square-zero}
\endeq
and near $z=\infty$:
\eq
(g'(z)-\phi'(z;y))^2=\left(-\frac{1}{4}z-\frac{1}{2}y-\kappa z^{-1}+\bo(z^{-2})\right)^2=\frac{1}{16}z^2+\frac{1}{4}yz + \left(\frac{1}{4}y^2+\frac{1}{2}\kappa\right) + \bo(z^{-1}),\quad z\to\infty.
\label{eq:square-infinity}
\endeq
By Liouville's Theorem it therefore follows that if $\tfrac{1}{8}E$ denotes the common coefficient of $z^{-1}$ in \eqref{eq:square-zero} and \eqref{eq:square-infinity}, then
\eq
\begin{split}
(g'(z)-\phi'(z;y))^2&=\frac{1}{16}z^2+\frac{1}{4}yz+\left(\frac{1}{4}y^2+\frac{1}{2}\kappa\right)+\frac{1}{8}Ez^{-1}+z^{-2}\\
&=\frac{1}{16z^2}P(z),
\end{split}
\label{eq:spectral-curve}
\endeq
where \emph{$P(\cdot)$ is exactly the same quartic polynomial defined in \eqref{eq:elliptic-ODE} under the identification $y=y_0$}.

There are, in principle, five possible configurations for the quartic $P(z)$, only three of which are consistent with our assumptions:  
\begin{itemize}
\item[$\{31\}$:] \emph{Two distinct roots, one of multiplicity $3$ and one simple.}  Suppose that $P(z)=(z-\alpha)^3(z-\beta)$ for some $\alpha\neq\beta$.  Comparing the coefficients and eliminating $\alpha$ and $\beta$ shows that given $\kappa\approx\kappa_\infty$ with $\kappa_\infty\in (-1,1)$,
$y=y_0$ must be a root of the $8^\mathrm{th}$ degree polynomial \eqref{eq:branch-points} defining the branch points of equilibrium solutions $U_0$ (see \eqref{eq:equilibrium}, and Proposition~\ref{prop:triangles}; these are precisely the vertices visible in the plots in Figure~\ref{fig:Domains}).  Given $\kappa$, for each of these eight points, the values of $\alpha$, $\beta$, and $E$ are uniquely determined.
\item[$\{211\}$:] \emph{Three distinct roots, one double and two simple.}  Suppose that $P(z)=(z-\alpha)(z-\beta)(z-\gamma)^2$ for distinct values $\alpha$, $\beta$, and $\gamma$.  Comparing the coefficients yields the system of equations
\eq
\begin{split}\alpha+\beta+2\gamma&=-4y\\
\alpha\beta+2(\alpha+\beta)\gamma +\gamma^2&=4(y^2+2\kappa)\\
2\alpha\beta\gamma + (\alpha+\beta)\gamma^2&=-2E\\
\alpha\beta\gamma^2&=16.
\end{split}
\label{eq:case-iv-coefficient-match}
\endeq
Eliminating $\alpha$ and $\beta$ between the first, second, and fourth equations gives the following quartic equation for $\gamma$:
\eq
Q(\gamma,y;\kappa):=\gamma^4+\frac{8}{3}y\gamma^3+\frac{4}{3}(y^2+2\kappa)\gamma^2-\frac{16}{3}=0.
\label{eq:gamma-eqn}
\endeq
This is precisely the same as the equation \eqref{eq:equilibrium} under the substitutions $y\mapsto y_0$ and $\gamma\mapsto U_0$, and hence the discriminant defining the branch points for $\gamma$ is \eqref{eq:branch-points} with $y_0$ replaced by $y$.  From \eqref{eq:equilibria-near-infinity} we see
that there are four distinct values of $\gamma$ when $y$ is large, namely
\begin{equation}
\gamma=2y^{-1}+\bo(y^{-2}),\quad\gamma=-2y^{-1}+\bo(y^{-2}),\quad\gamma=-2y+\bo(1),\quad\text{and}\quad
\gamma=-\tfrac{2}{3}y+\bo(1),\quad y\to\infty.
\end{equation}
Given $y\in\mathbb{C}$, $\kappa\approx\kappa_\infty$ with $\kappa_\infty\in (-1,1)$, 
and any root $\gamma$ of \eqref{eq:gamma-eqn}, the values of $E$, $\alpha$, and $\beta$ (the latter up to permutation) are determined from 
\begin{equation}
E=-16\gamma^{-1}+2y\gamma^2+\gamma^3,\quad\alpha\beta=16\gamma^{-2},\quad\alpha+\beta=-4y-2\gamma.
\label{eq:alpha-beta-E}
\end{equation}
According to Theorems~\ref{thm:HermiteExterior}--\ref{thm:OkamotoExterior} in light of Remark~\ref{rem:equilibrium-rederive}, this configuration will turn out to be relevant for $y_0\in\mathcal{E}_\mathrm{F}(\kappa)$ depending on the family $\mathrm{F}=\mathrm{gH}$ or $\mathrm{F}=\mathrm{gO}$.
\item[$\{1111\}$:] \emph{Four distinct roots, all simple.}  Suppose that $P(z)=(z-\alpha)(z-\beta)(z-\gamma)(z-\delta)$ for distinct values $\alpha,\beta,\gamma,\delta$.  This configuration again places no conditions on $y$ or $\kappa$, but now the constant $E$ is also free, and some additional conditions need to be specified to relate it to $y$ and $\kappa$.  These will be the \emph{Boutroux conditions} to be introduced later (see also \eqref{eq:intro-Boutroux}). This configuration will turn out to be relevant for $y_0$ in the bounded regions $\rectangle(\kappa)$, $\pm\TR(\kappa)$, $\pm\TI(\kappa)$ introduced in Section~\ref{sec:intro-results-nonequilibrium}.
\item[$\{4\}$:] \emph{One root of multiplicity $4$.}  Suppose that $P(z)=(z-\alpha)^4$ for some $\alpha$.  Comparing the coefficients and eliminating $\alpha$ shows easily that this form is consistent only if $y^4=16$ and $\kappa=\tfrac{1}{4}y^2$.  But further eliminating $y^2=\pm 4$ gives $\kappa=\pm 1$ which is inconsistent in the limit with $\kappa\to\kappa_\infty$ with $\kappa_\infty\in (-1,1)$.
\item[$\{22\}$:] \emph{Two distinct double roots.}  Suppose that $P(z)=(z-\alpha)^2(z-\beta)^2$ for some $\alpha\neq\beta$.  Comparing the coefficients then yields, as in the case of a single root of multiplicity $4$, that $\kappa=\pm 1$ which is inconsistent for large $T$ with $\kappa_\infty\in (-1,1)$.  
\end{itemize}
Therefore, only cases $\{31\}$, $\{211\}$, and $\{1111\}$ will be relevant to our study going forward, and we will say that the spectral curve is of \emph{class} $\{31\}$, $\{211\}$, or $\{1111\}$.  In any of these cases, we can solve for $g'(z)$ by introducing suitable bounded branch cuts between pairs of distinct roots of the quartic $P(z)$ and defining a function $R(z)$ analytic except on these cuts that satisfies $R(z)^2=P(z)$ and $R(z)=z^2+\bo(z)$ as $z\to\infty$.  Then, in order to satisfy the necessary condition $g'(z)\to 0$ as $z\to\infty$ we need to take the square root in \eqref{eq:spectral-curve} precisely as follows:
\eq
g'(z)=\phi'(z;y)-\frac{1}{4}z^{-1}R(z).
\label{eq:gprime-R}
\endeq
That this formula also gives $g'(z)=-sz^{-1}+\bo(1)$ as $z\to 0$ then requires in addition that $R(0)=4s$, which we interpret as a condition on how the branch cuts of $R(z)$ must be placed relative to the origin in order to achieve the correct sign of $R(0)$.  In general, the function $g$ will depend parametrically on $y\in\mathbb{C}$ and the related parameters $s=\pm 1$ and $\kappa\approx\kappa_\infty$ with $\kappa_\infty\in (-1,1)$,
so when it is necessary to emphasize the parameter dependence we will write $g(z)=g^{(s,\kappa)}(z;y)$ going forward.  It will be convenient to define the related function $h(z)=h^{(s,\kappa)}(z;y)$ by
\eq
h^{(s,\kappa)}(z;y):=g^{(s,\kappa)}(z;y)-\phi(z;y)\quad\implies\quad
h^{(s,\kappa)\prime}(z;y)=-\frac{R(z)}{4z}.
\label{eq:h-g-phi}
\endeq

A key role in our analysis will be played by certain trajectories of the quadratic differential $h'(z)^2\,\dd z^2$, and how they depend on $y$ once the coefficient $E$ is suitably determined as a function of $y$ for given $\kappa$ and the family (gO or gH) of interest.  We will deduce all of the needed properties theoretically below, but it is also easy to compute them numerically, so as a preview of what will come, we present plots of the trajectories connected to simple roots of $P(z)$ in Figures~\ref{fig:Trajectories-gO-k0} and \ref{fig:Trajectories-gH-k0} (plots corresponding to exceptional values of $y_0$, indicated by a red dot on a curve in the central inset of each figure, show all trajectories connected to roots of $P(z)$ of any multiplicity).
\begin{figure}[h]
\hspace{-0.2in}
\begin{tabular}{c}
\includegraphics[height=1.2in]{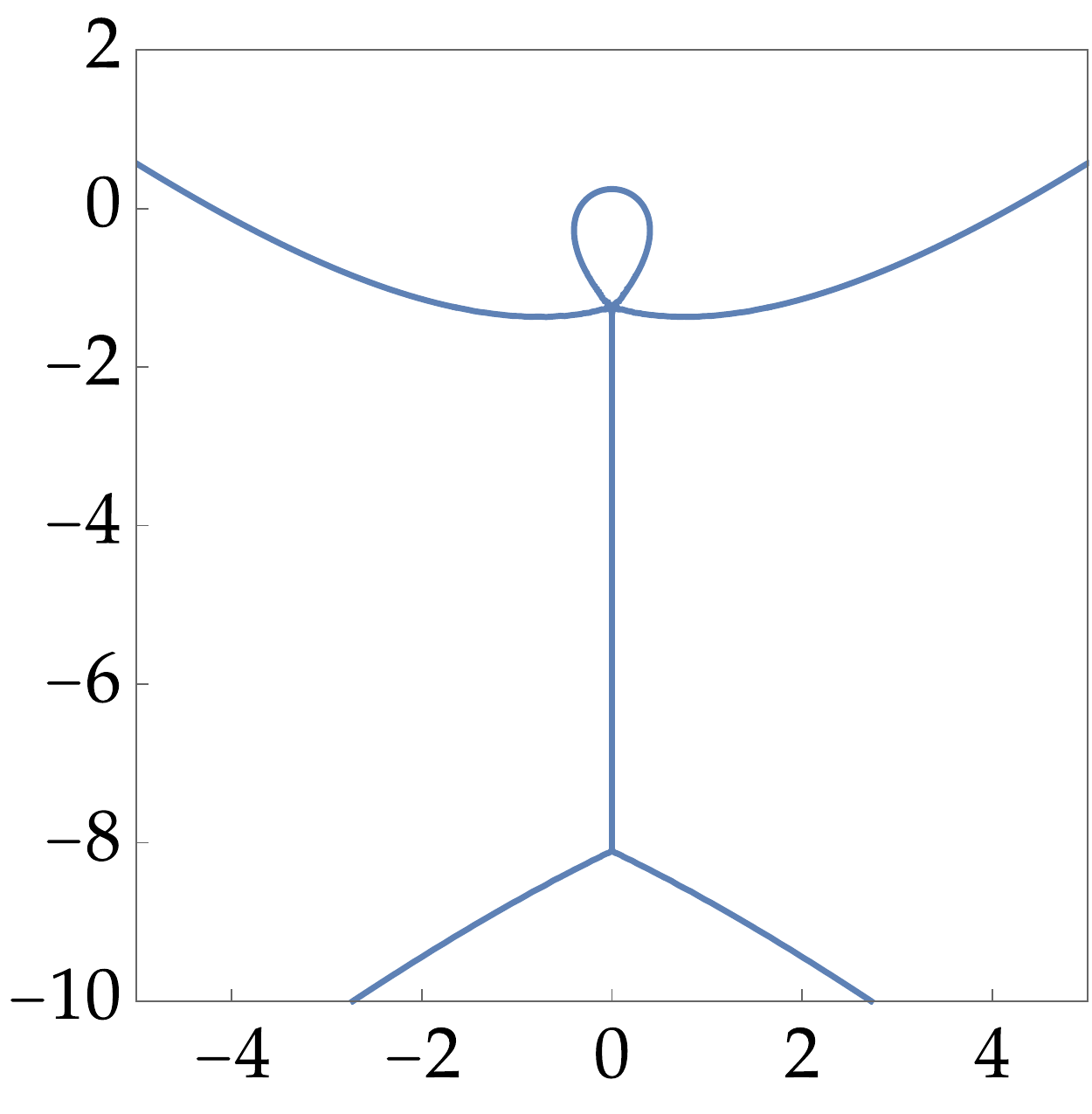}\\
\includegraphics[height=1.2in]{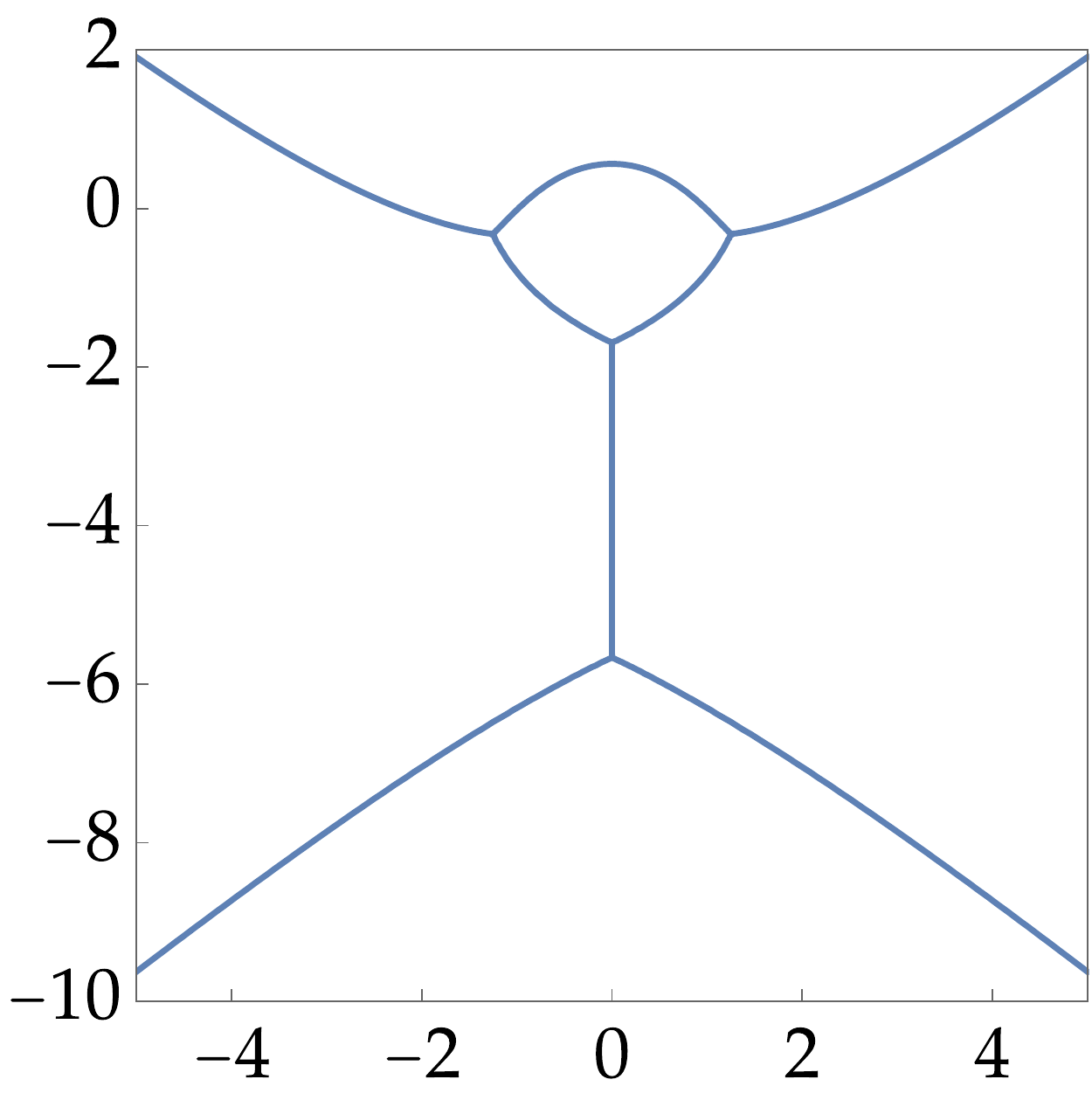}\\
\includegraphics[height=1.2in]{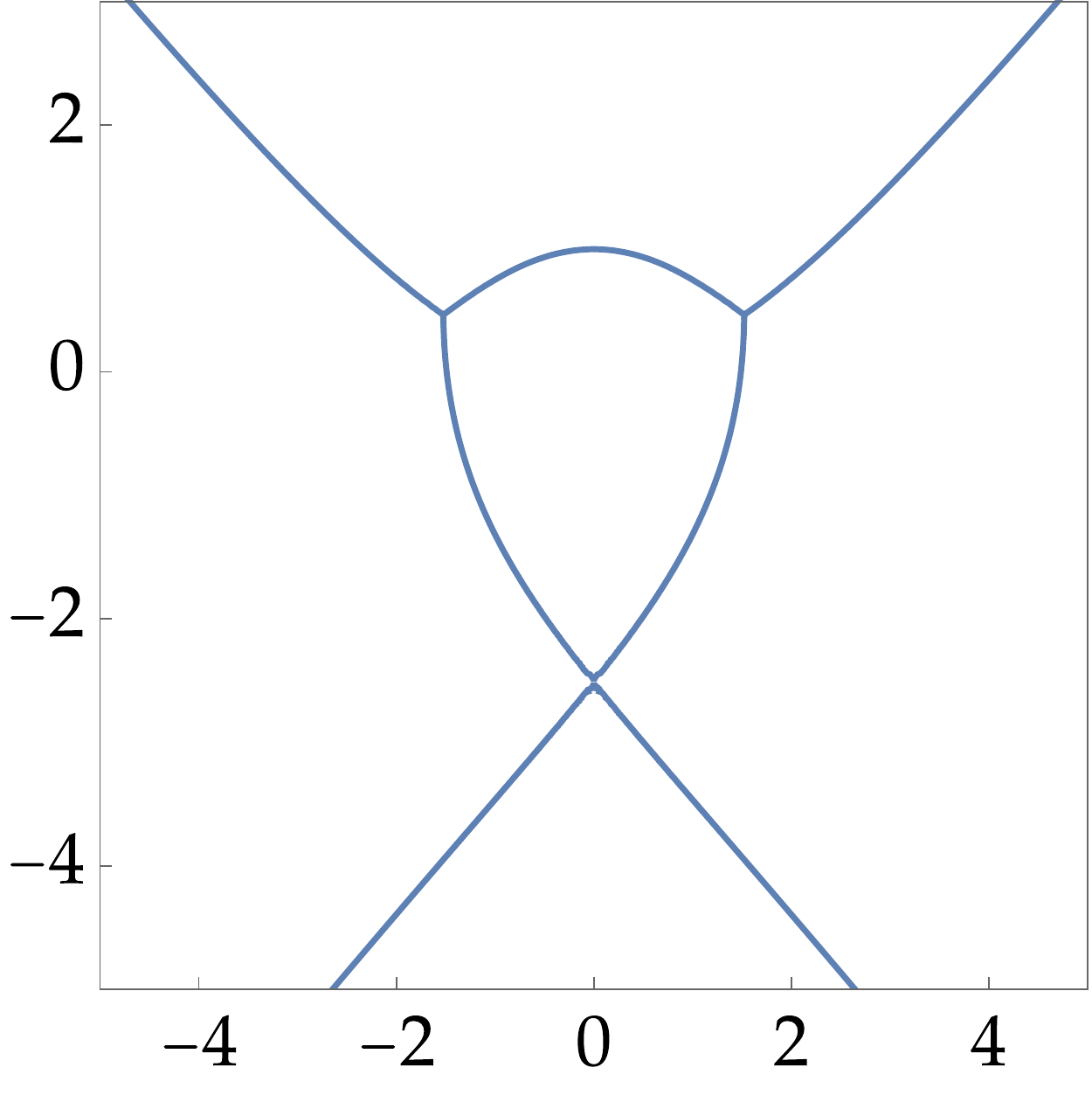}\\
\includegraphics[height=1.2in]{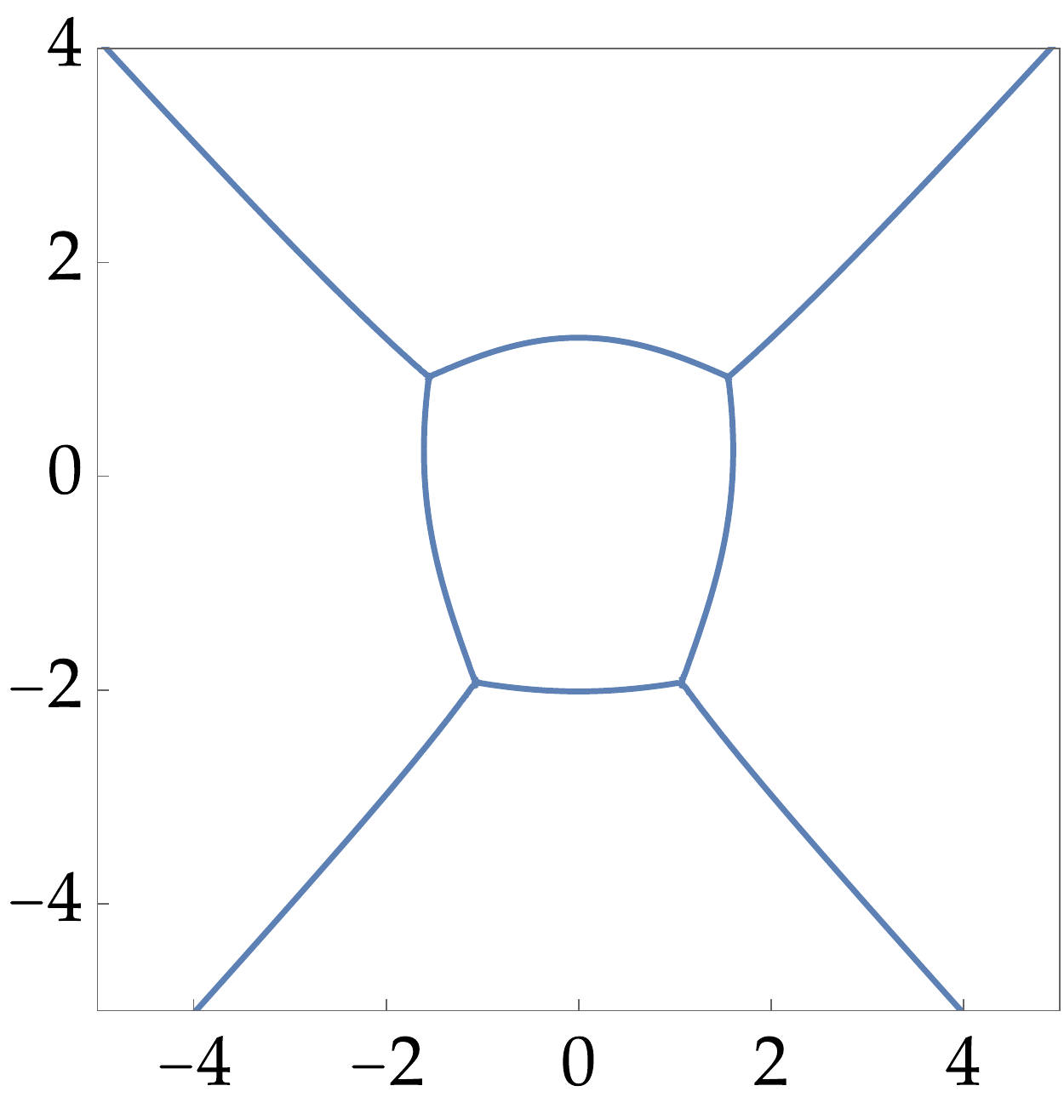}\\
\includegraphics[height=1.2in]{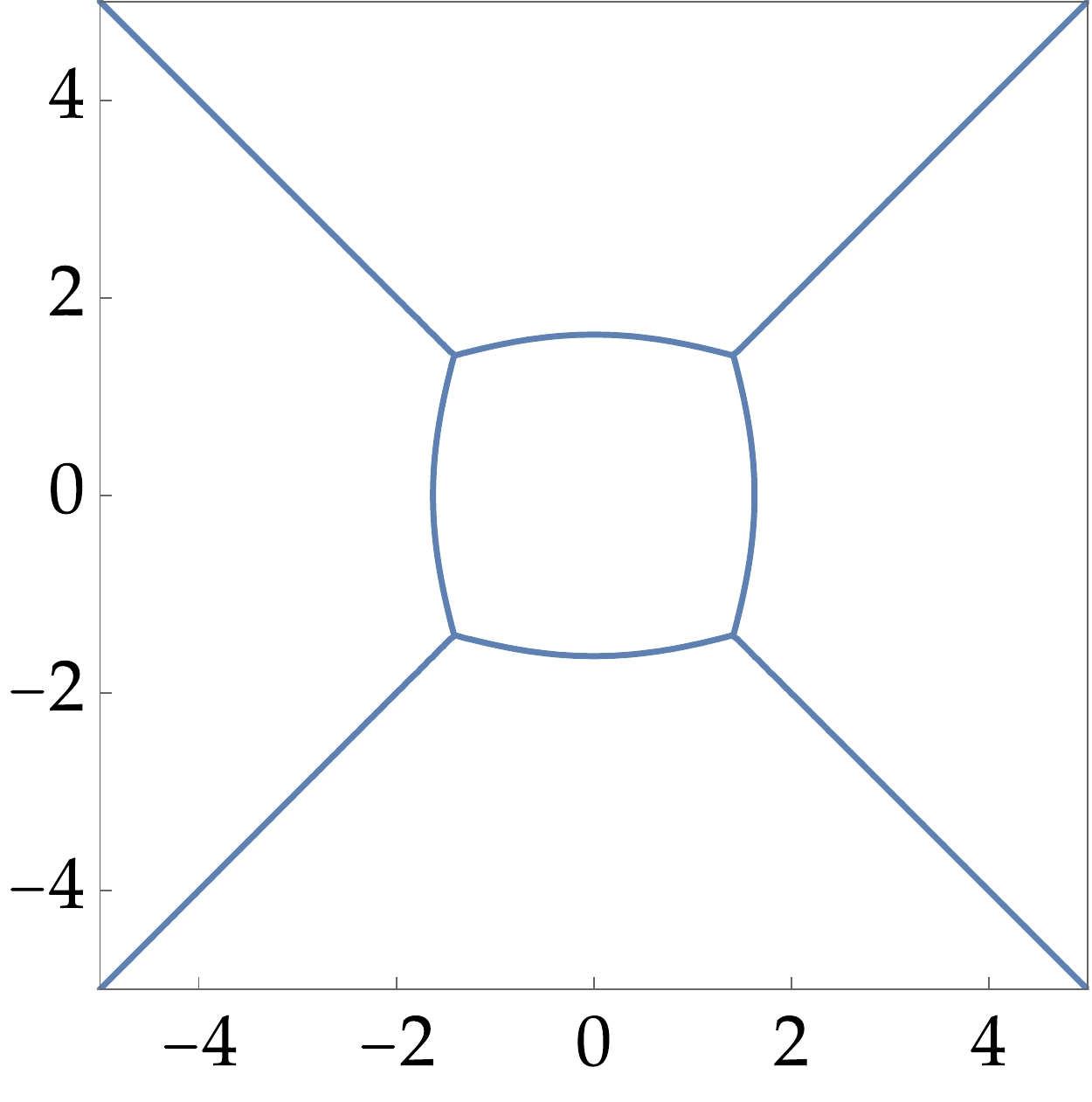}\\
\end{tabular}
\hspace{-.18in}
\begin{tabular}{c}
\includegraphics[height=1.2in]{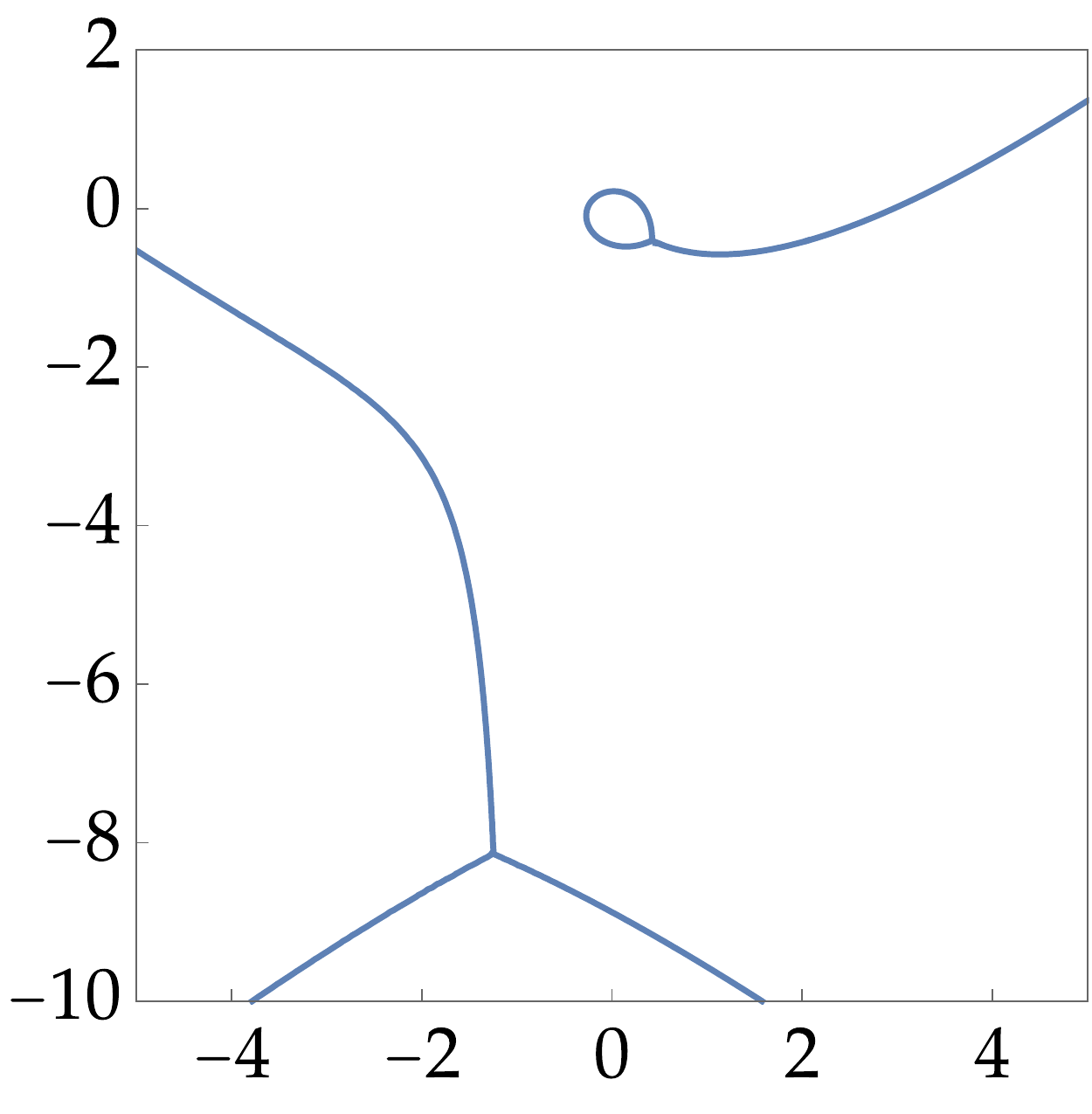}
\includegraphics[height=1.2in]{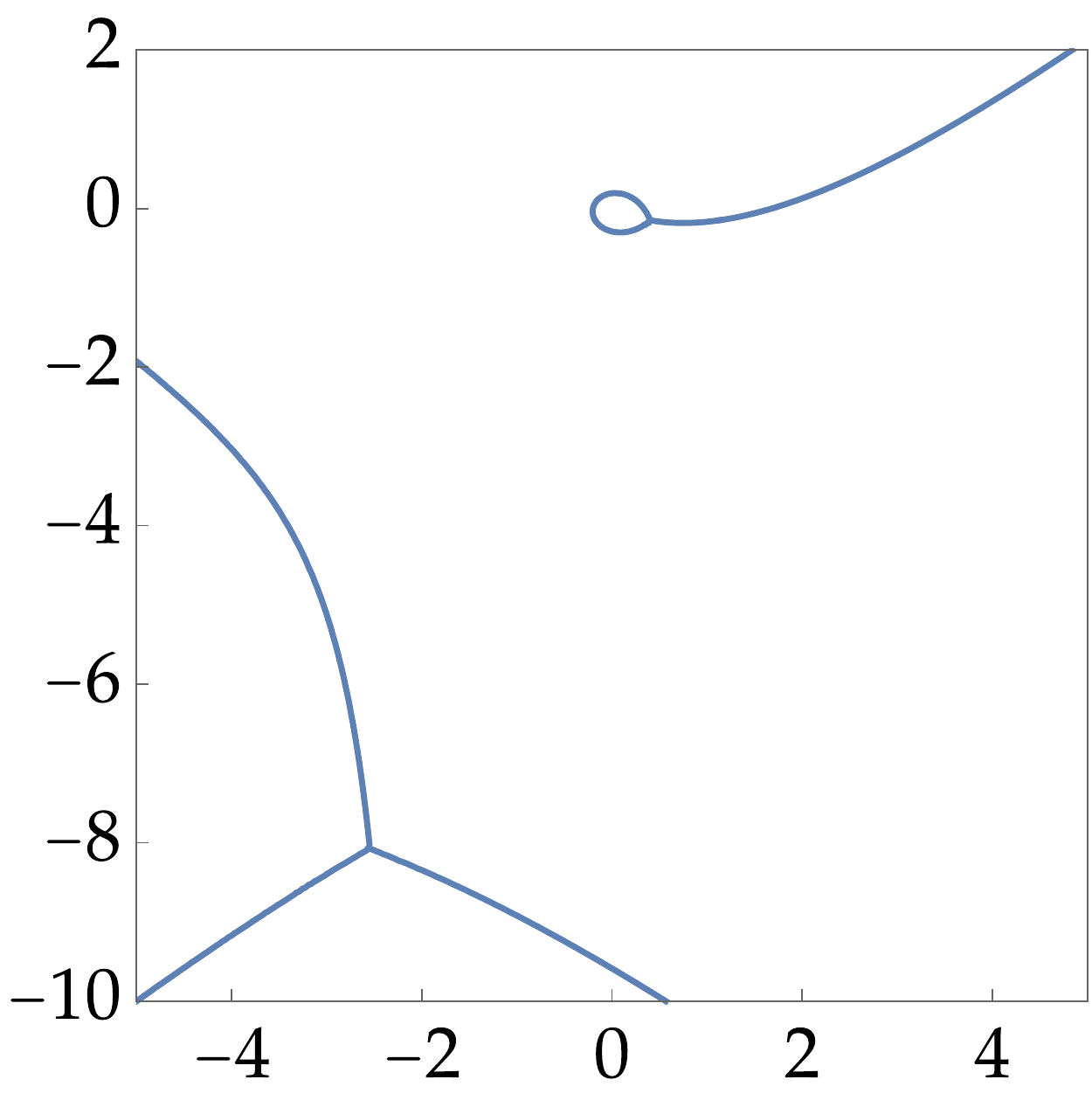}
\includegraphics[height=1.2in]{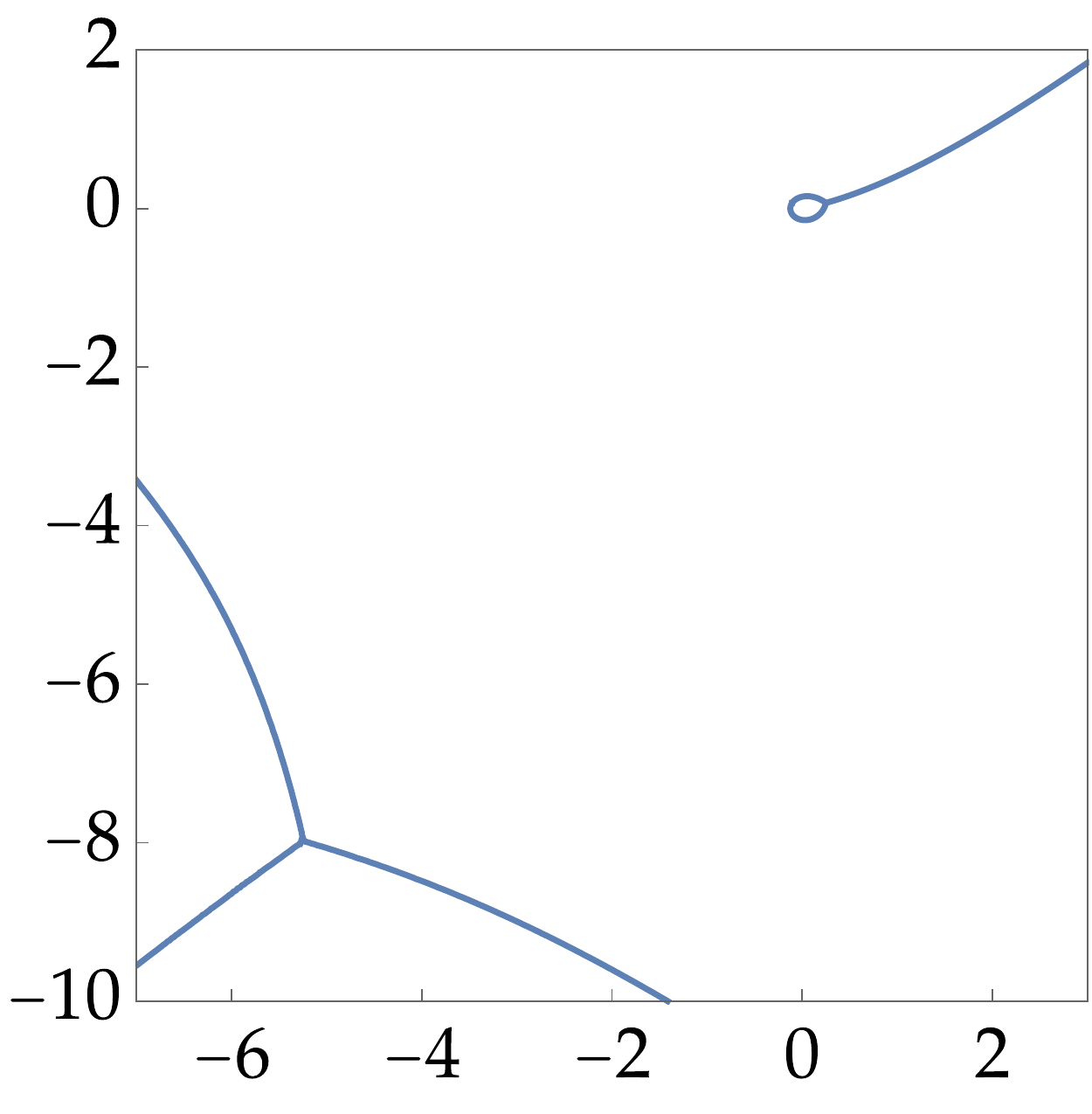}\\
\vspace{.18in}\\
\includegraphics[width=3in]{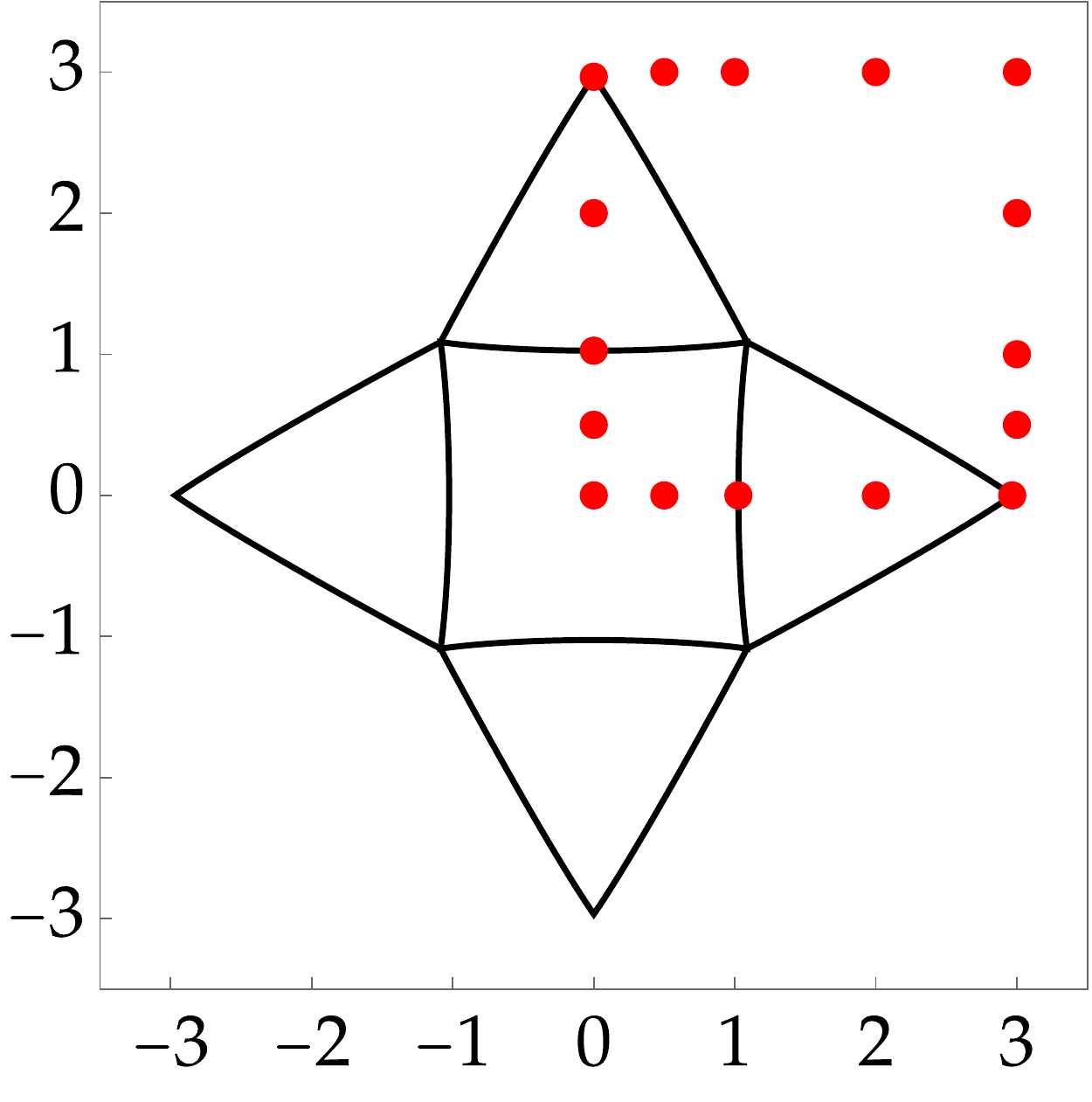}\\
\vspace{.18in}\\
\includegraphics[height=1.2in]{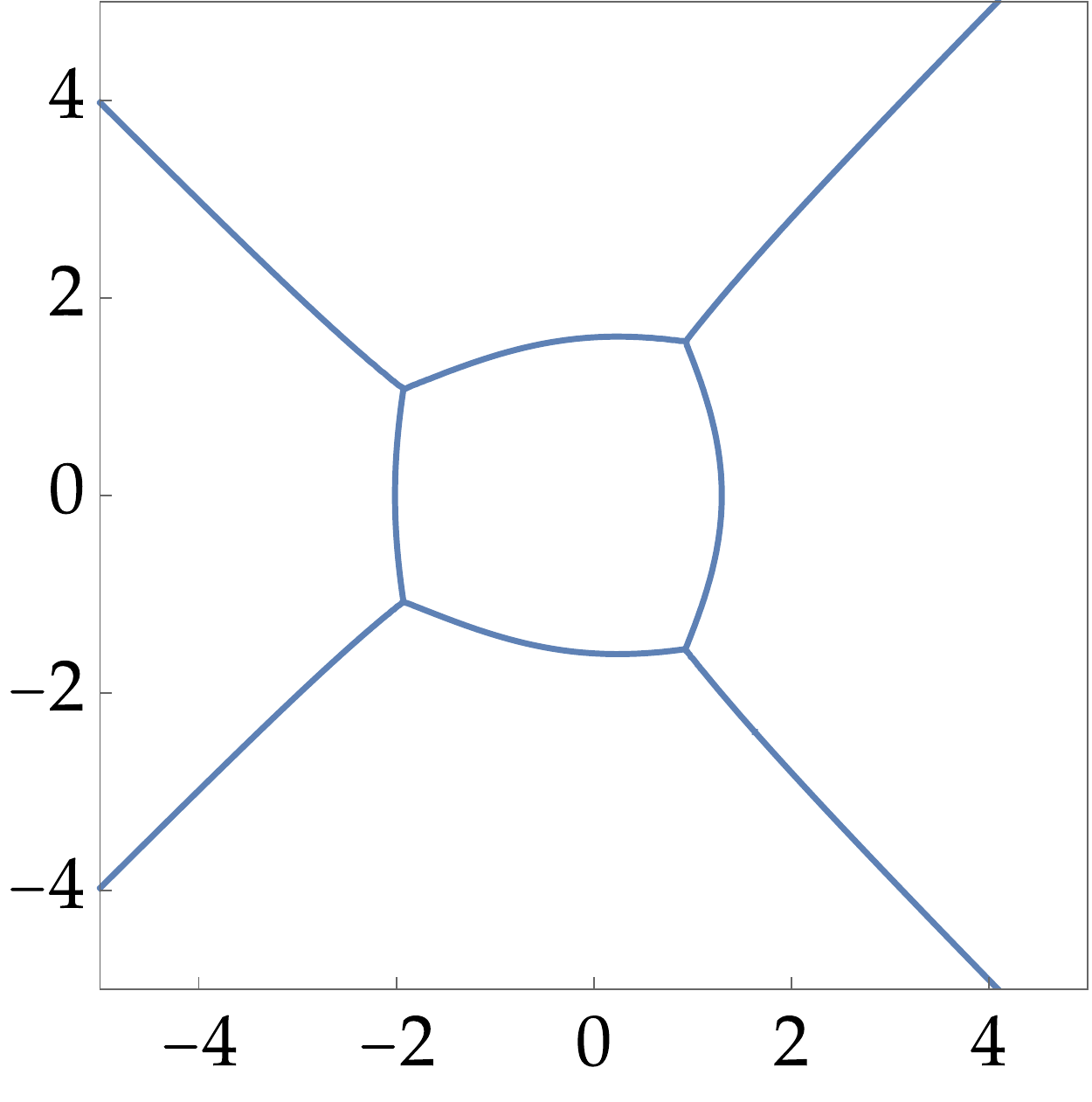}
\includegraphics[height=1.2in]{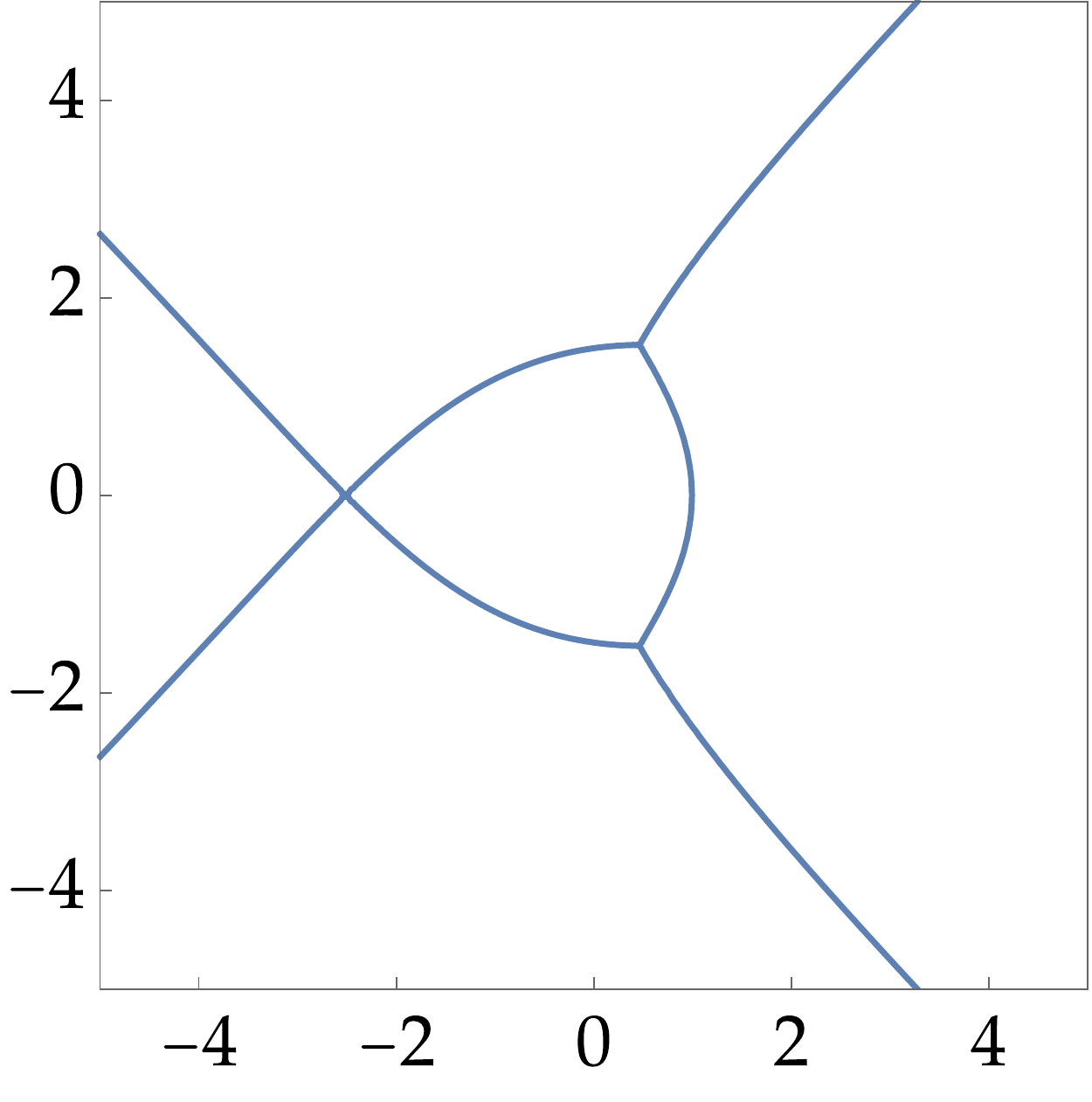}
\includegraphics[height=1.2in]{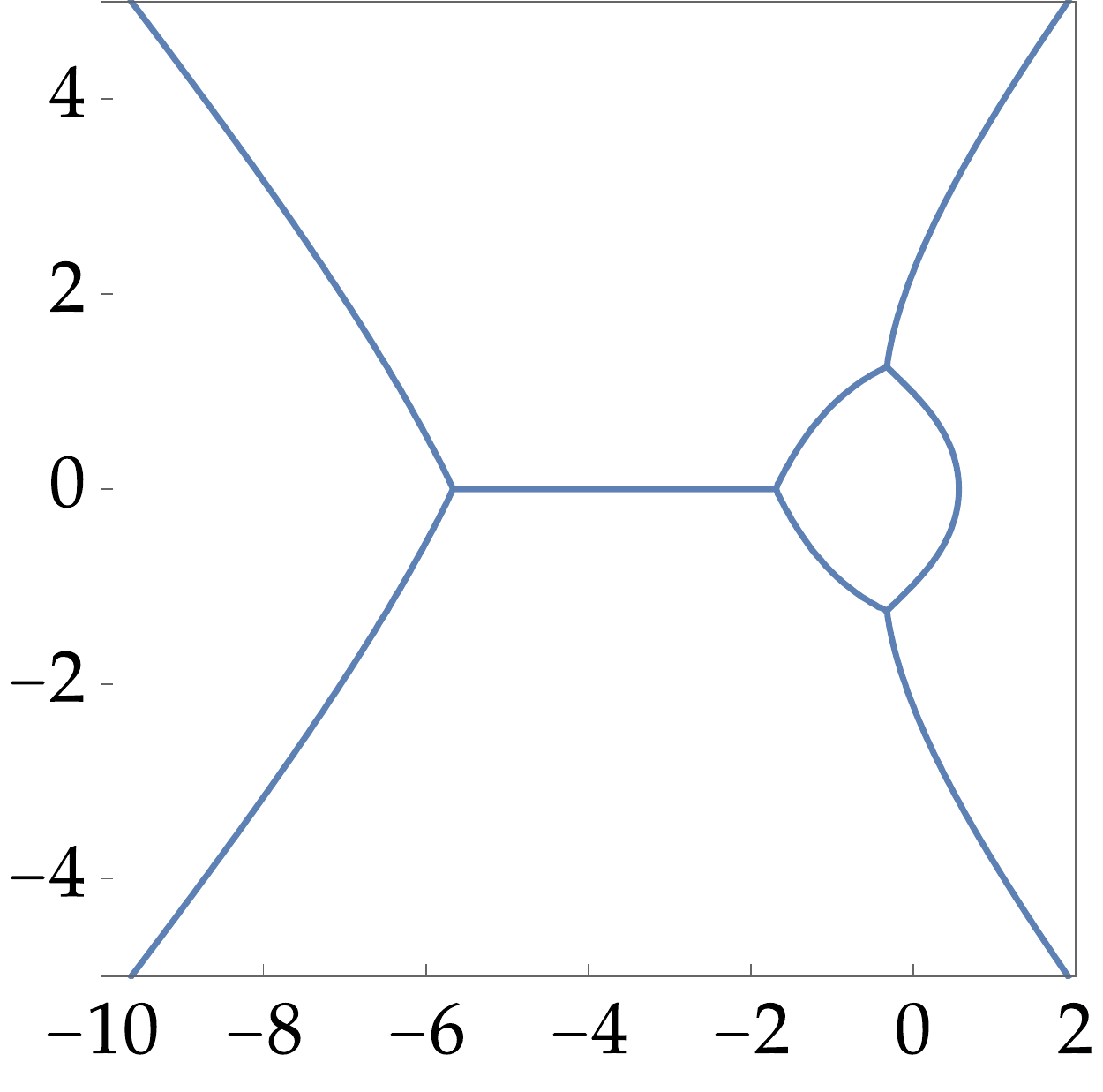}\\
\end{tabular}
\hspace{-.18in}
\begin{tabular}{c}
\includegraphics[height=1.2in]{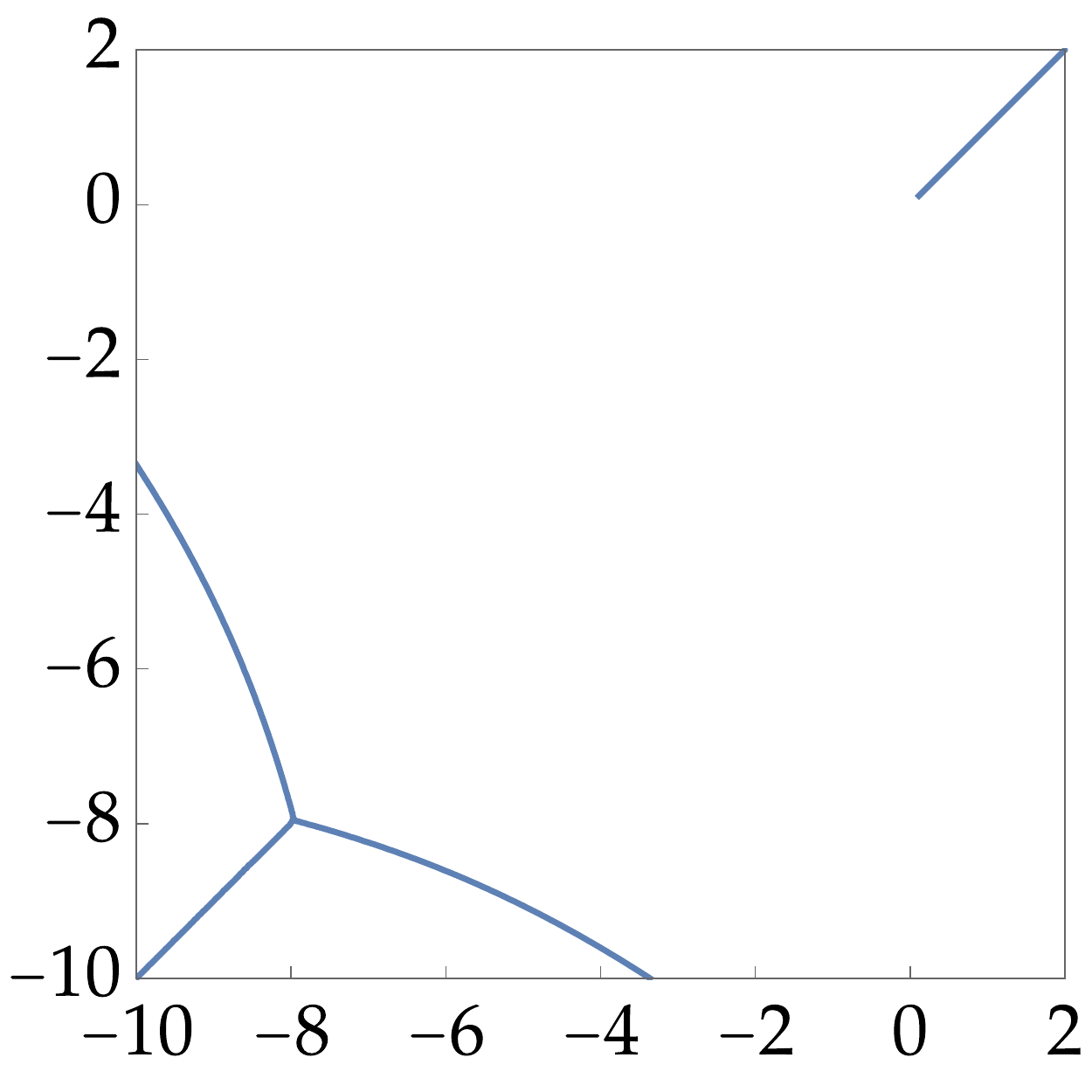}\\
\includegraphics[height=1.2in]{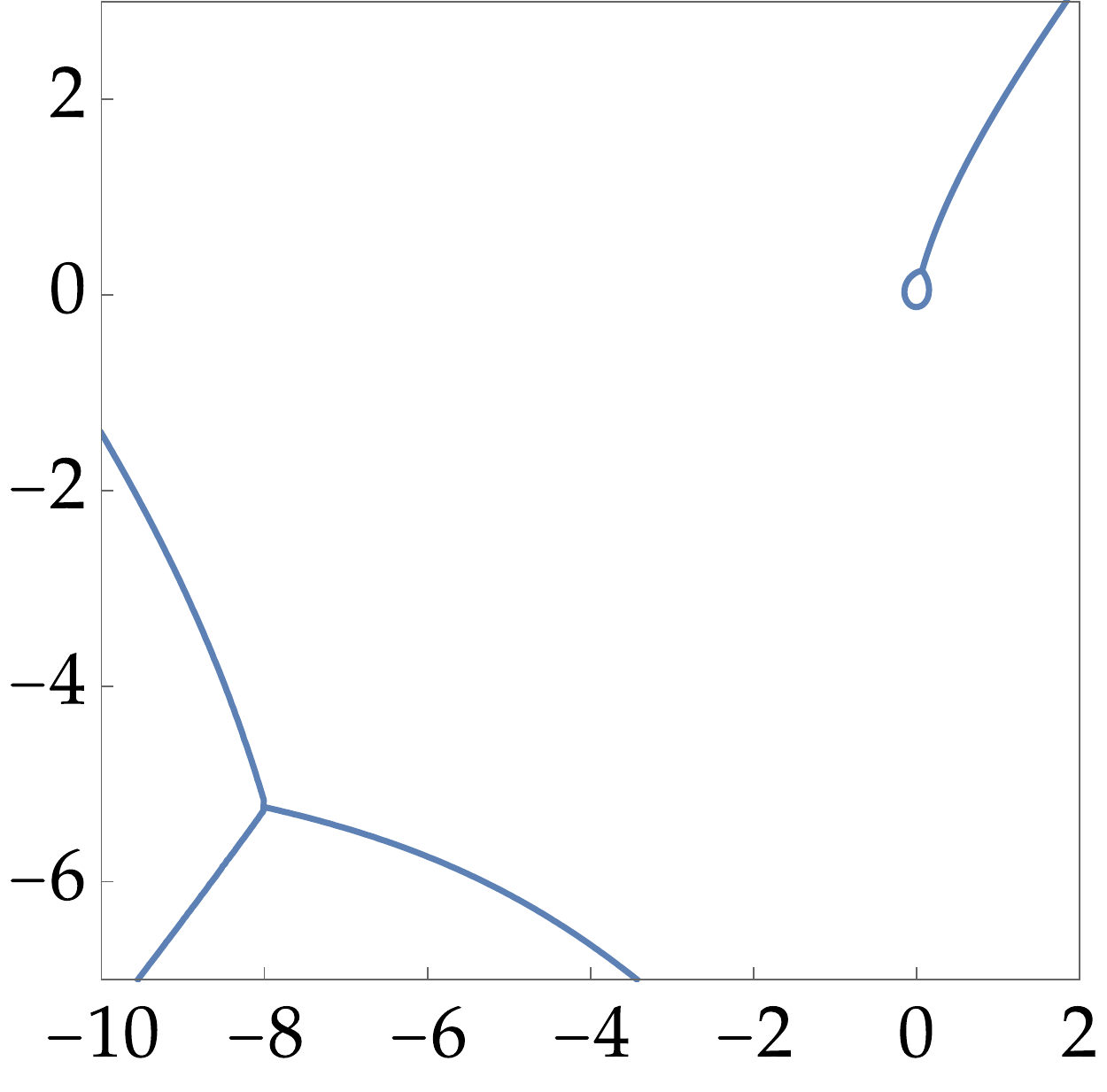}\\
\includegraphics[height=1.2in]{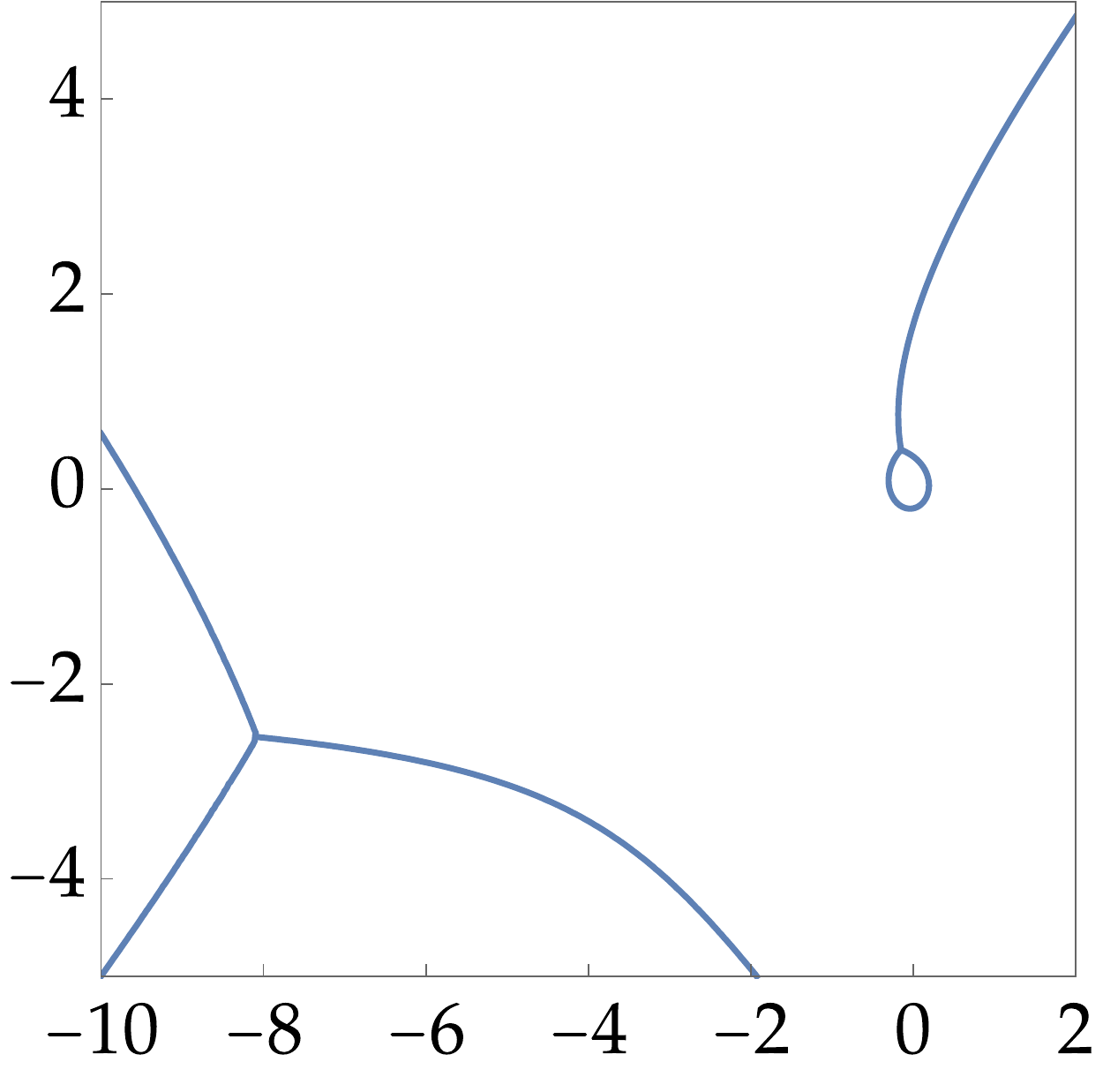}\\
\includegraphics[height=1.2in]{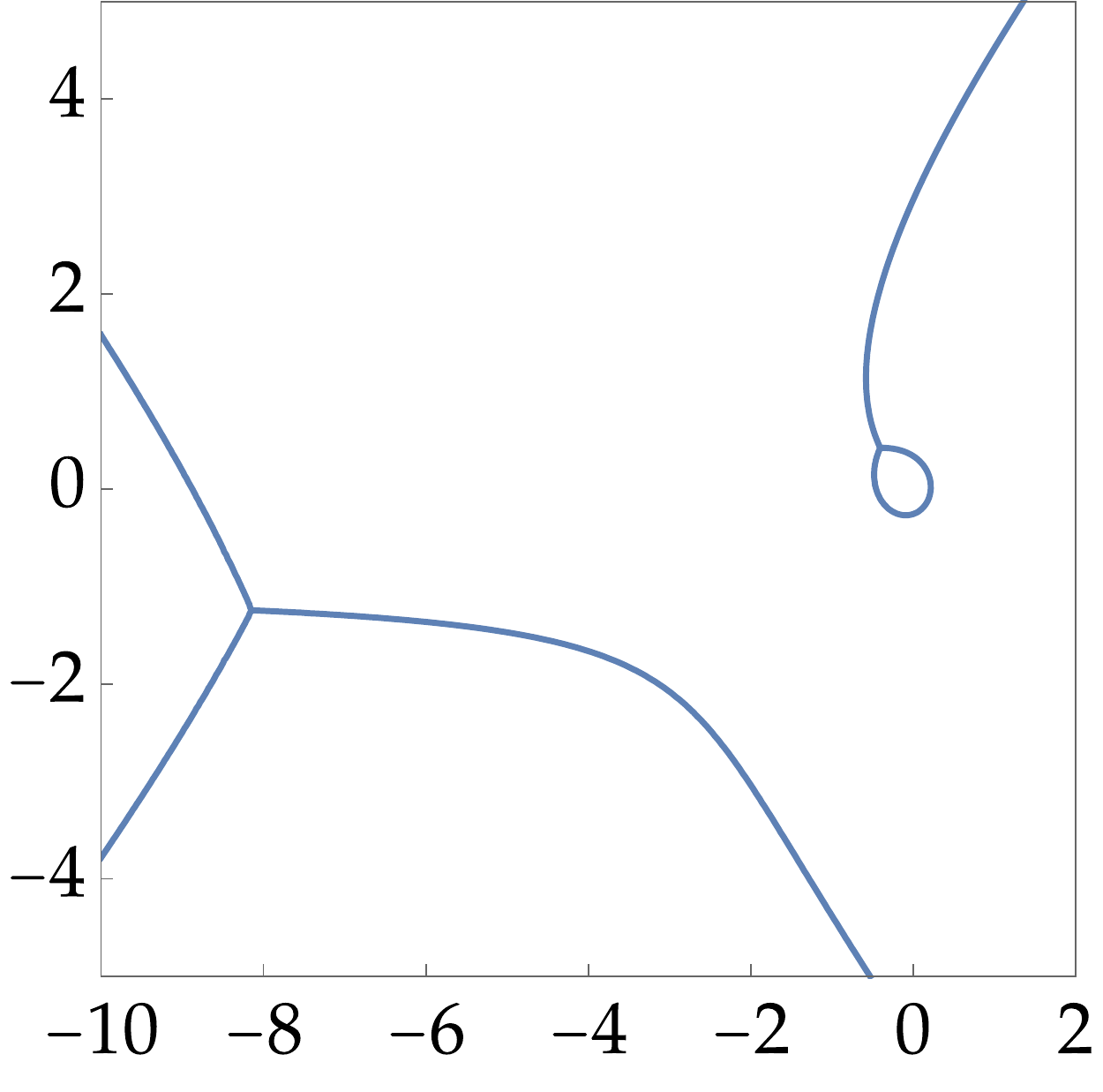}\\
\includegraphics[height=1.2in]{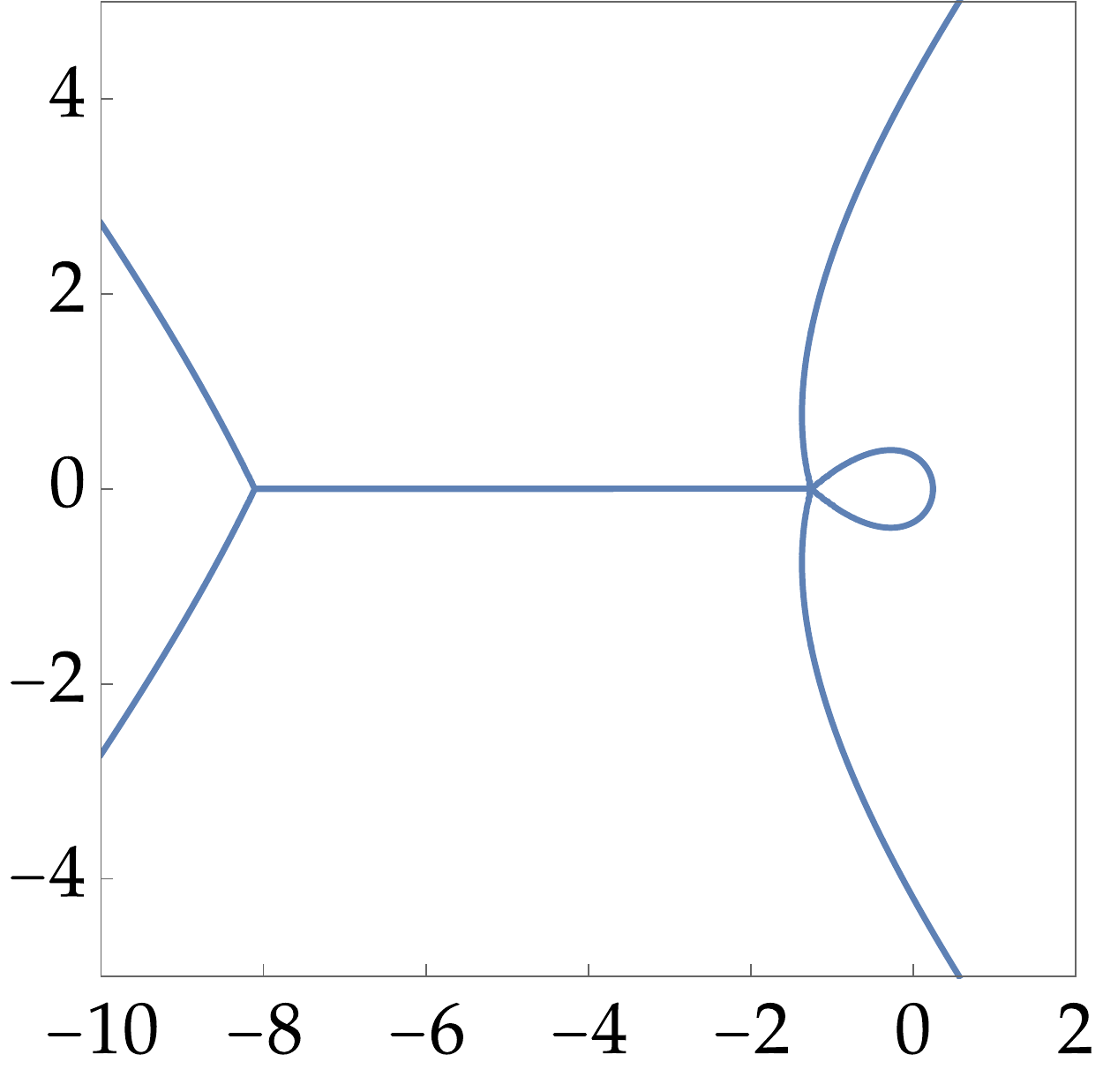}\\
\end{tabular}
\caption{Critical trajectories of $h'(z)^2\,\dd z^2$ emanating from (for generic $y_0$) simple roots of $P(z)$ for $\kappa=0$ for the gO family.  The same topological structure holds for $-1<\kappa<1$.  Counterclockwise from top left:  $y\approx 2.96773i$, $y=2i$, $y\approx 1.0253i$, $y=0.5i$, $y=0$, $y=0.5$, $y\approx 1.0253$, $y=2$, $y\approx 2.96773$, $y=3+0.5i$, $y=3+i$, $y=3+2i$, $y=3+3i$, $y=2+3i$, $y=1+3i$, $y=0.5+3i$.  Inset:  Boundaries of the regions $\rectangle(0)$, $\pm\TR(0)$, and $\pm\TI(0)$ in the $y$-plane.  The $y$-values corresponding to different trajectory plots are indicated by red dots.}
\label{fig:Trajectories-gO-k0}
\end{figure}

\begin{figure}[h]
\hspace{-0.2in}
\begin{tabular}{c}
\includegraphics[height=1.2in]{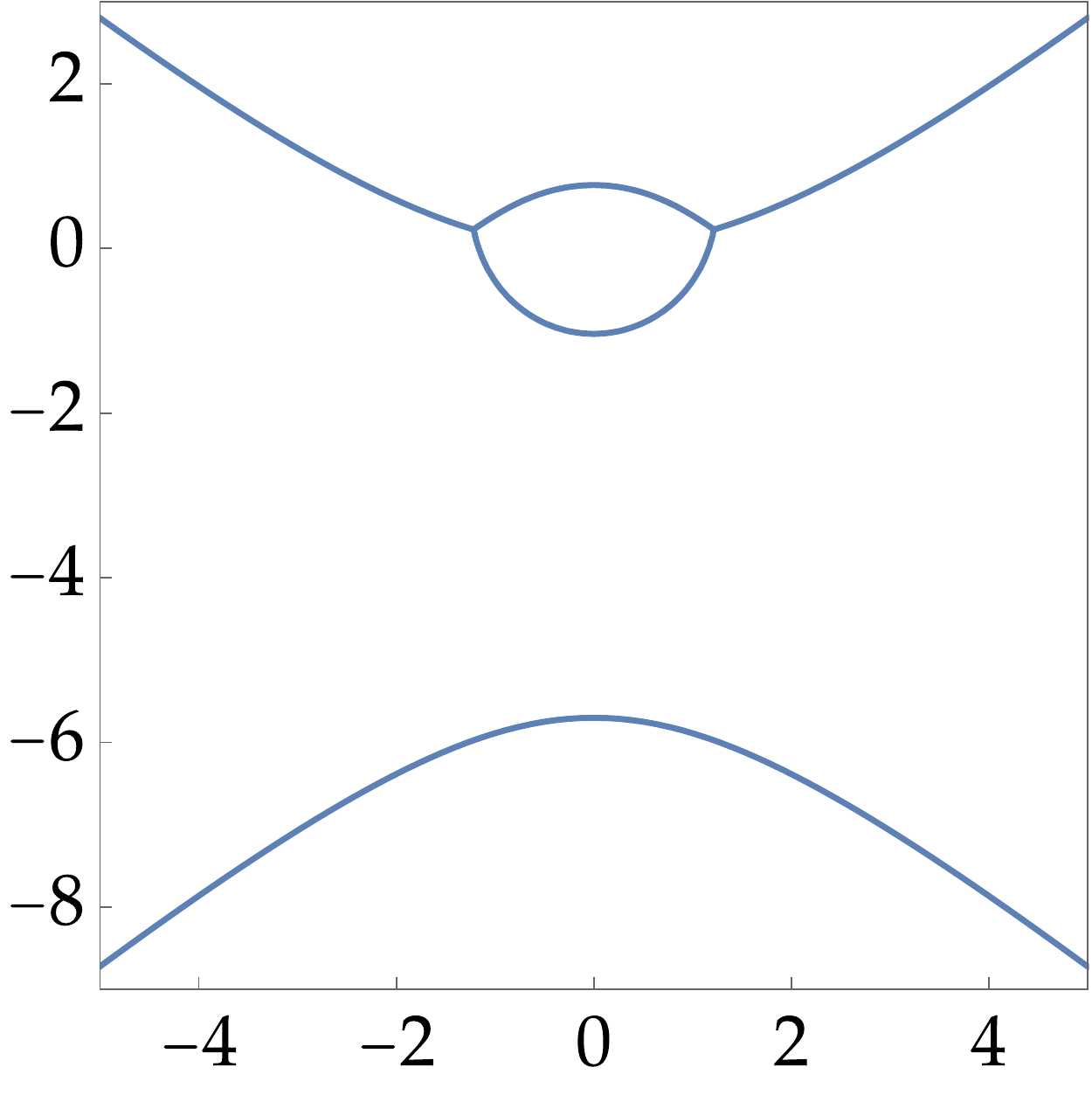}\\
\includegraphics[height=1.2in]{stokes-k0-y1p0253i.pdf}\\
\includegraphics[height=1.2in]{stokes-k0-y0p5i.pdf}\\
\includegraphics[height=1.2in]{stokes-k0-y0.pdf}\\
\end{tabular}
\hspace{-.18in}
\begin{tabular}{c}
\includegraphics[height=1.2in]{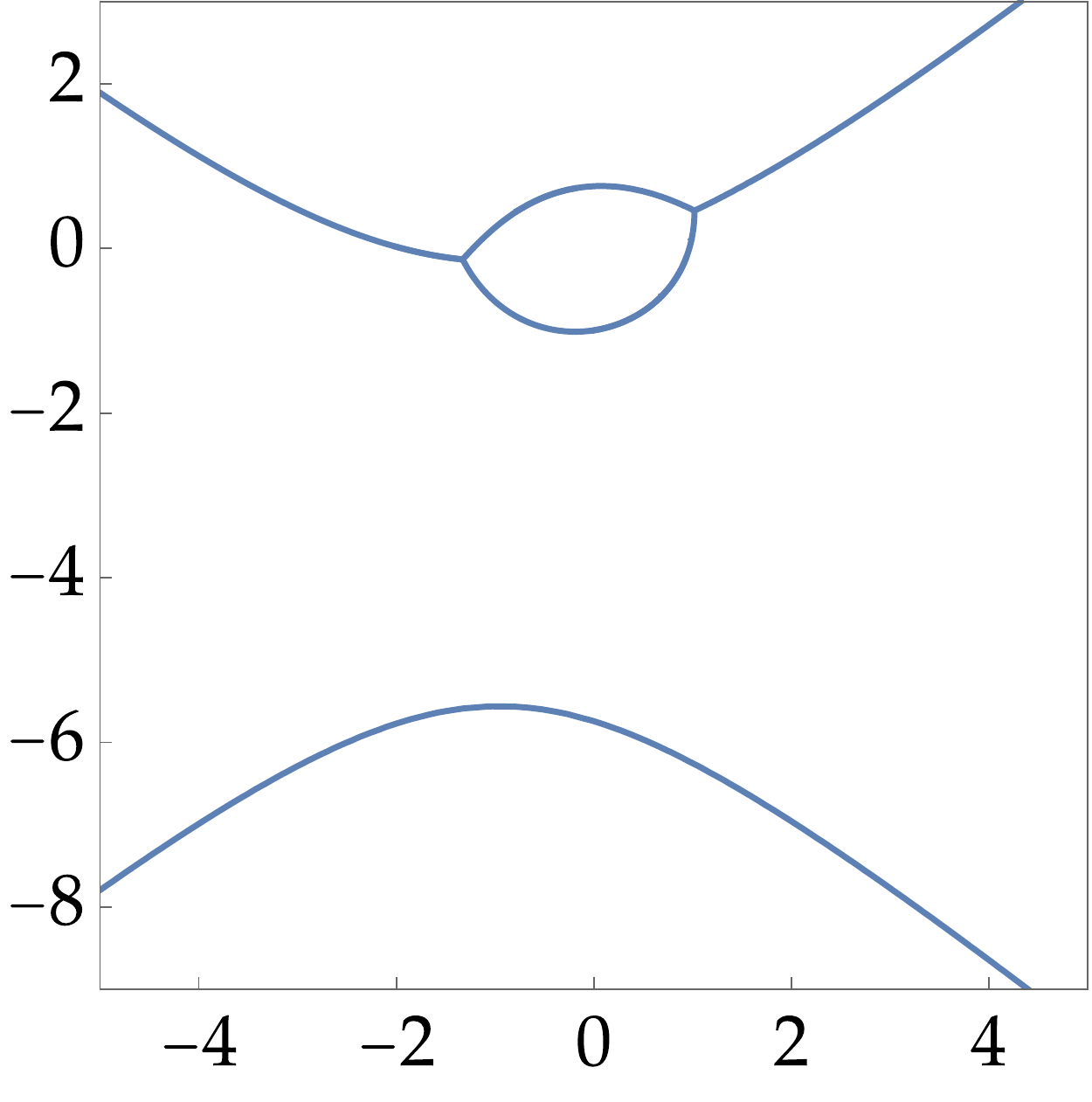}
\includegraphics[height=1.2in]{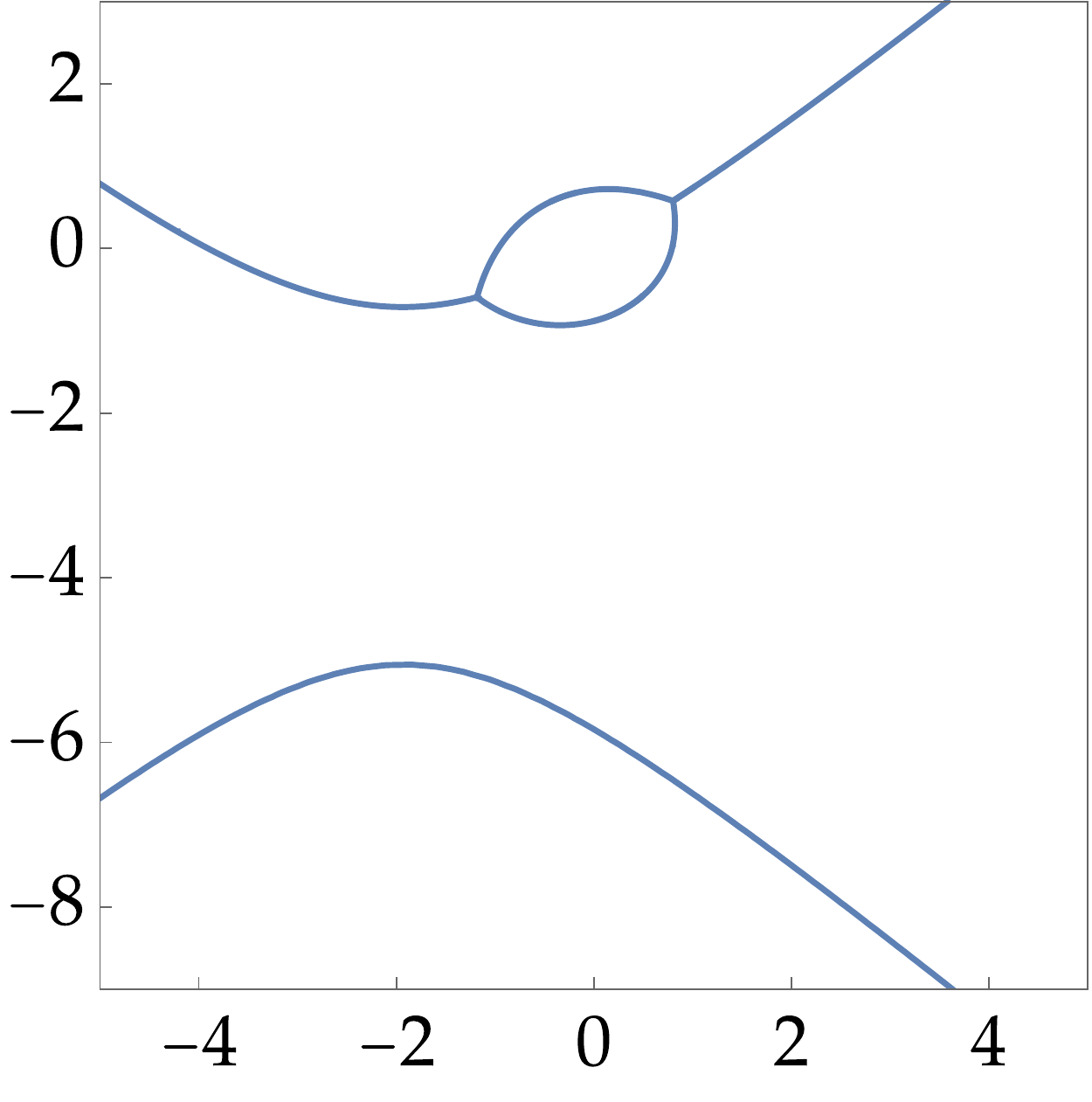}\\
\vspace{-.13in}\\
\includegraphics[width=2.3in]{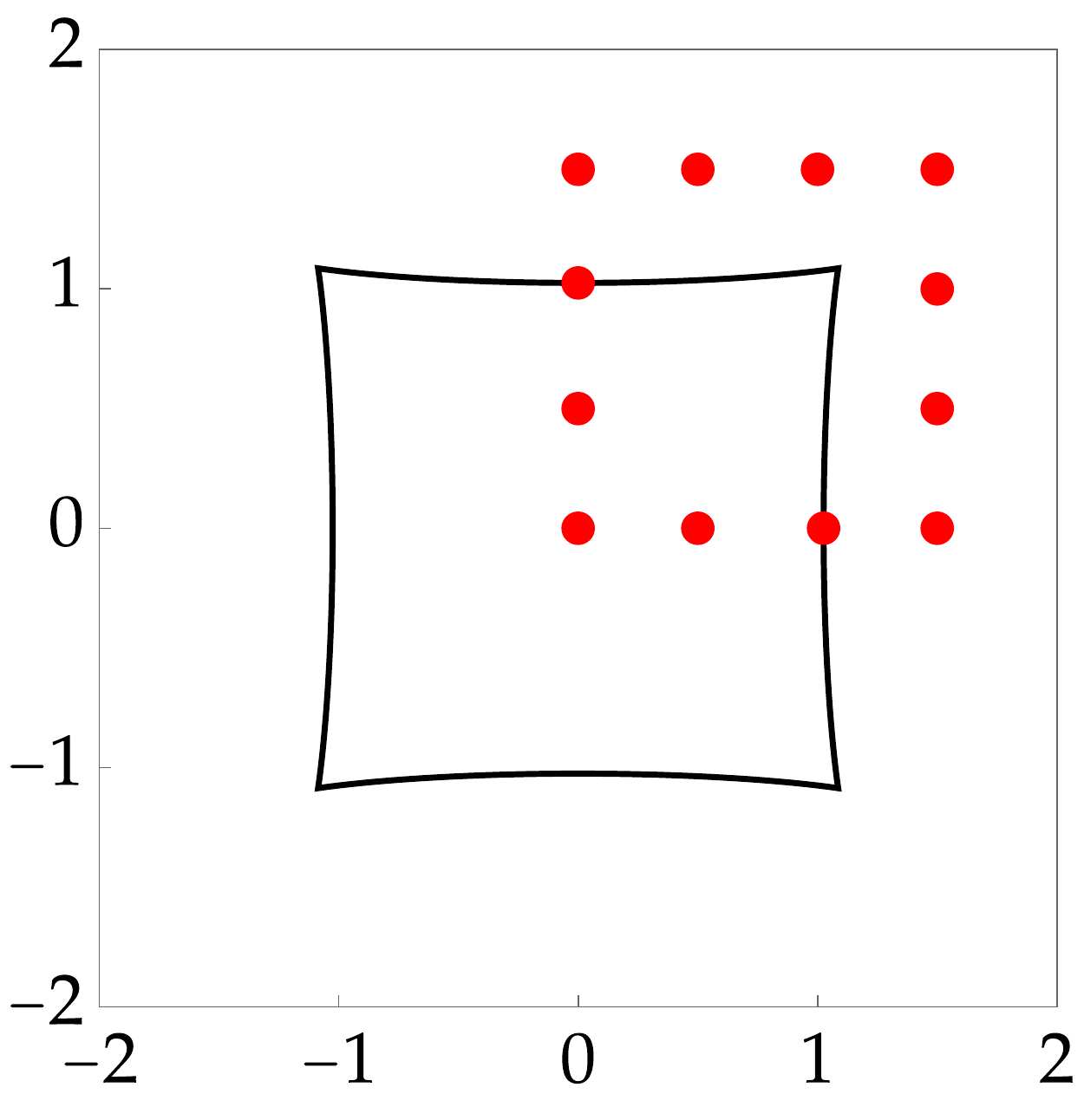}\\
\vspace{-.13in}\\
\includegraphics[height=1.2in]{stokes-k0-y0p5.pdf}
\includegraphics[height=1.2in]{stokes-k0-y1p0253.pdf}\\
\end{tabular}
\hspace{-.18in}
\begin{tabular}{c}
\includegraphics[height=1.2in]{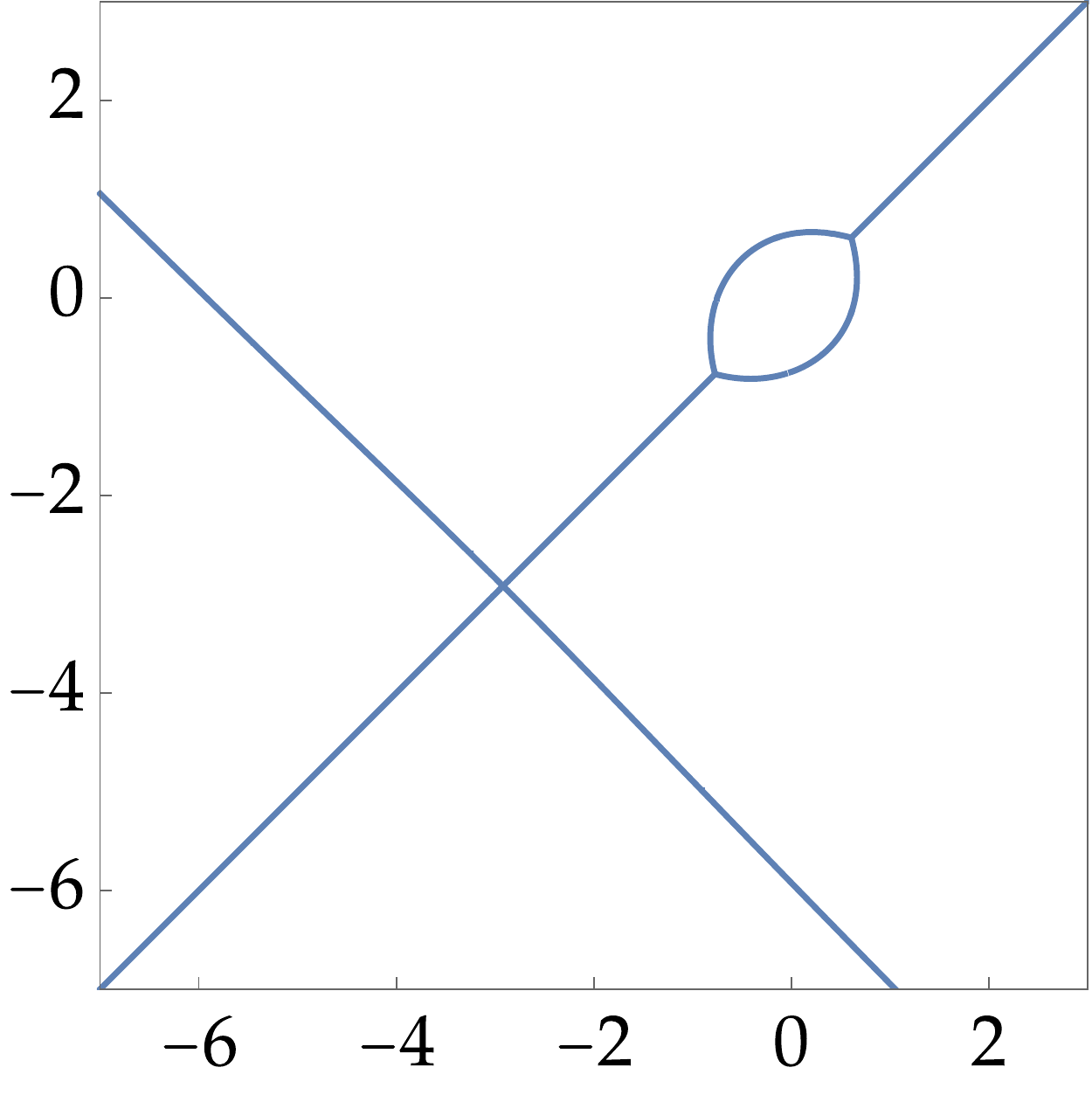}\\
\includegraphics[height=1.2in]{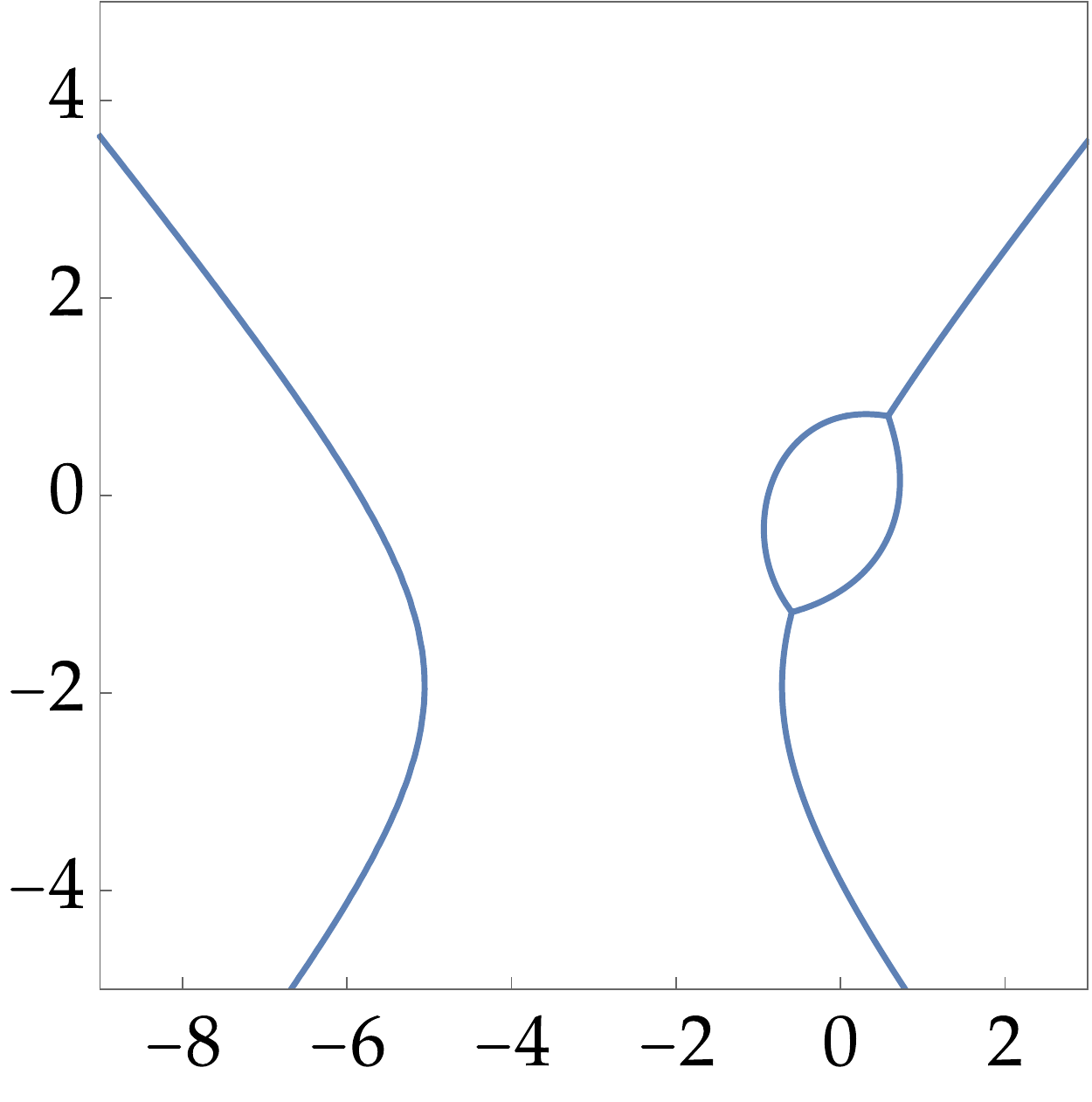}\\
\includegraphics[height=1.2in]{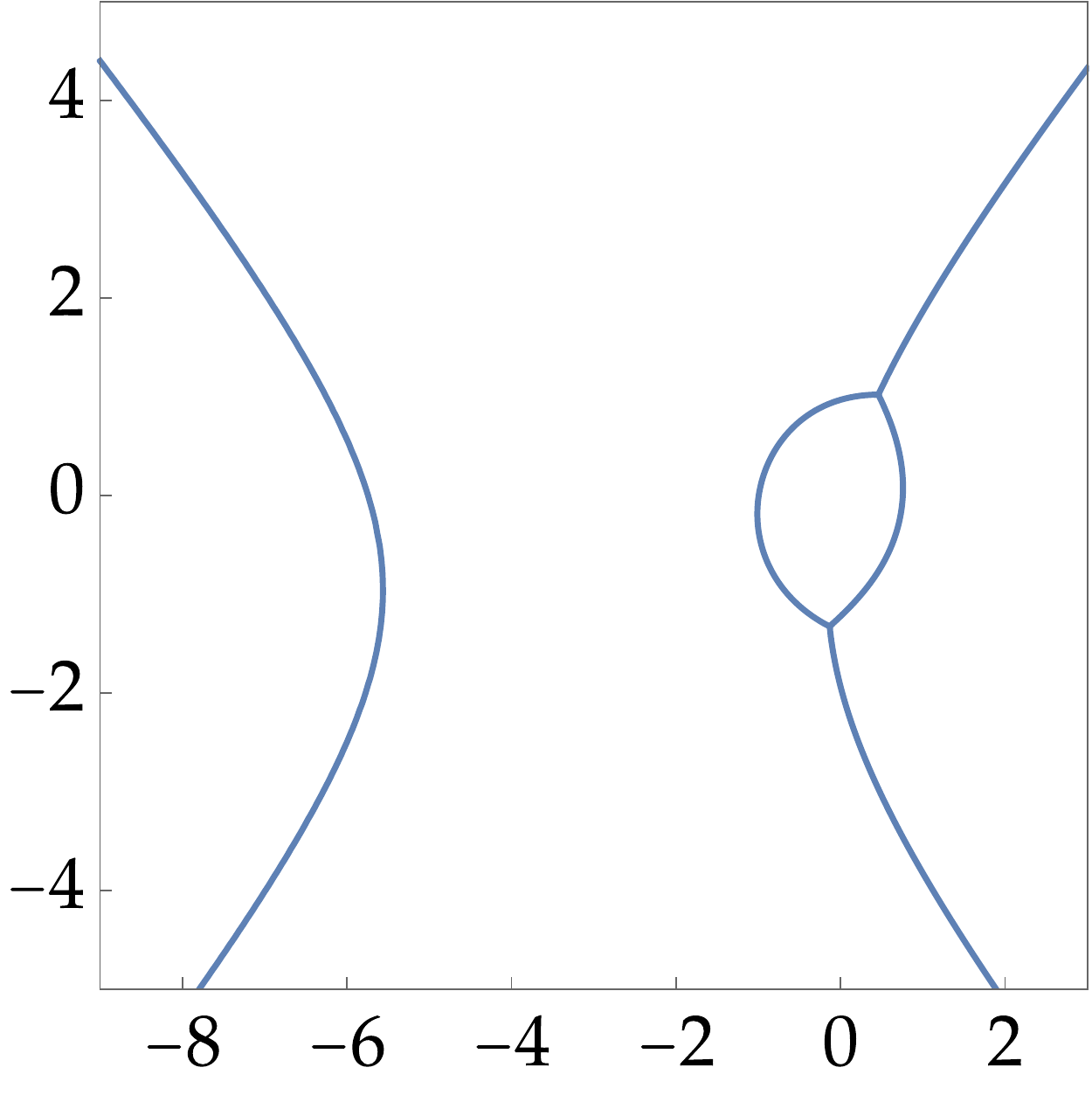}\\
\includegraphics[height=1.2in]{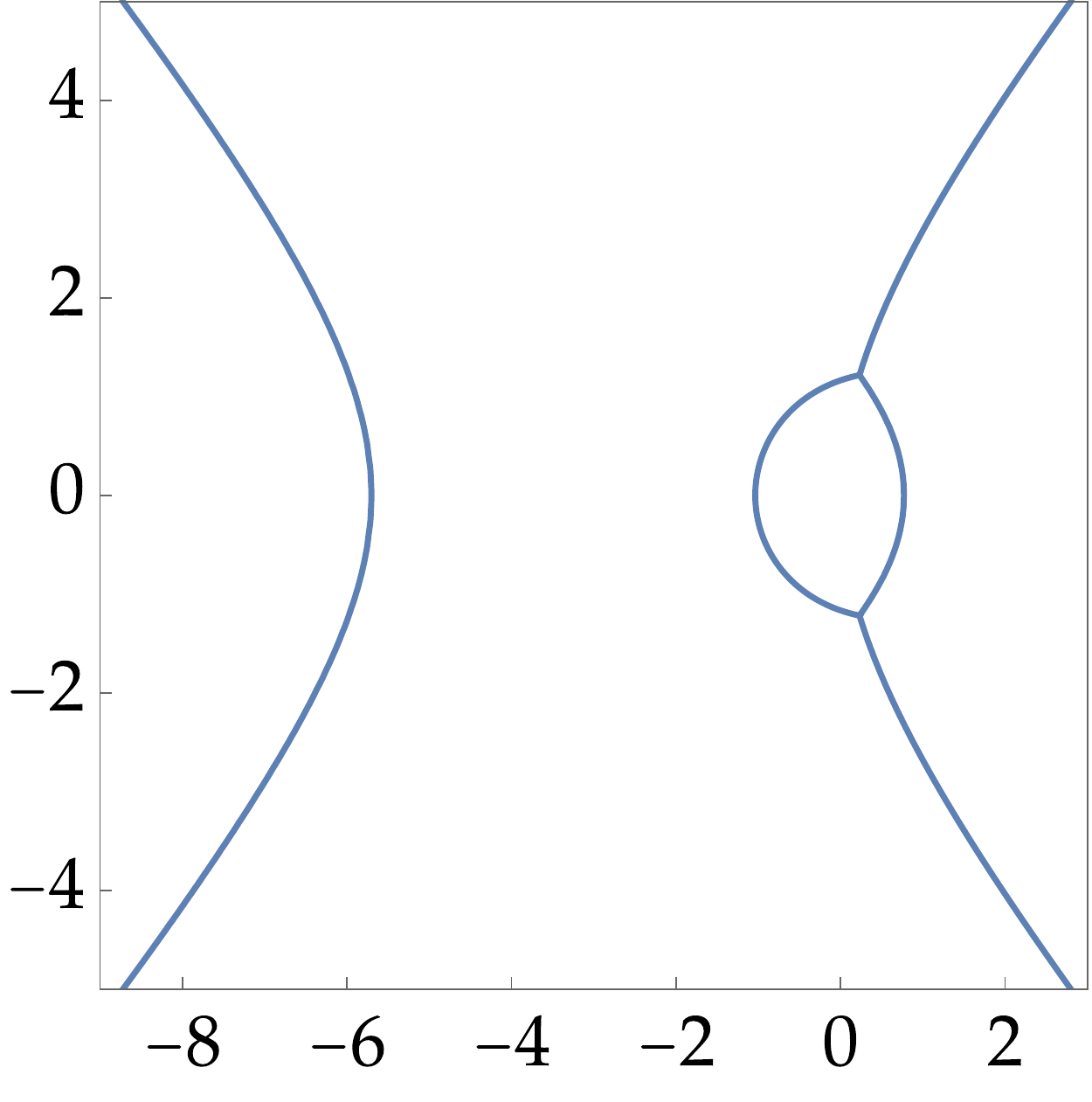}
\end{tabular}
\caption{Critical trajectories of $h'(z)^2\,\dd z^2$ emanating from (for generic $y_0$) simple roots of $P(z)$ for $\kappa=0$ for the gH family.  The same topological structure holds for $-1<\kappa<1$.  Counterclockwise from top left:  $y=1.5i$, $y\approx 1.0253i$, $y=0.5i$, $y=0$, $y=0.5$, $y\approx 1.0253$, $y=1.5$, $y=1.5+0.5i$, $y=1.5+i$, $y=1.5+1.5i$, $y=1+1.5i$, $y=0.5+1.5i$.  Inset:  Boundary of $\rectangle(0)$ in the $y$-plane.  The $y$-values corresponding to different trajectory plots are indicated by red dots. }
\label{fig:Trajectories-gH-k0}
\end{figure}

\section{Asymptotic analysis of $\mathbf{M}(z)$ for sufficiently large $|y|$:  gO case}
\label{sec:Exterior}
Fix $\kappa_\infty\in (-1,1)$, so that also $\kappa$ is bounded away from $\pm 1$ for $T$ sufficiently large.  We assume that for $|y|$ sufficiently large with $0\le\arg(y)\le\tfrac{1}{2}\pi$, the polynomial $P(z)$ is in case $\{211\}$ and we select the solution of the quartic \eqref{eq:gamma-eqn} that satisfies $\gamma=U_{0,\mathrm{gO}}(y;\kappa)=-\tfrac{2}{3}y+\bo(1)$ as $y\to\infty$ (see Section~\ref{sec:equilibrium}).  This solution can be analytically continued to a certain domain that we will describe more precisely later.  Corresponding to the asymptotic $\gamma=-\tfrac{2}{3}y+\bo(1)$, from \eqref{eq:case-iv-coefficient-match} we have (breaking permutation symmetry) $\alpha=-\tfrac{8}{3}y+\bo(1)$ and $\beta=-\tfrac{27}{2}y^{-3}+\bo(y^{-4})$ as $y\to\infty$.  

\begin{remark}
In Sections~\ref{sec:Exterior}, \ref{sec:Exterior-gH}, and \ref{sec:OutsideDomain}, we will set $\zeta=0$ and therefore identify $y_0$ with $y$.
\end{remark}

\subsection{Analysis of the exponent $h(z)$}
\label{sec:h-analysis}
The function $z\mapsto h'(z)=h^{(s,\kappa)\prime}(z;y)$ is well-defined for $z\neq 0$ up to a sign, and moreover the formula $Q(z):=h'(z)^2=\tfrac{1}{16}z^{-2}P(z)=\tfrac{1}{16}z^{-2}(z-\gamma)^2(z-\alpha)(z-\beta)$ shows that $h'(z)$ is meromorphic on a two-sheeted Riemann surface $\mathcal{R}$ over the $z$-plane having genus zero (a single branch cut connects $\alpha$ and $\beta$), with simple poles over $z=0$ and triple poles over $z=\infty$.  It is easily checked that since $\kappa\in\mathbb{R}$, the residues at all four poles are purely real, so since $\mathcal{R}$ has genus zero it follows that by integration that $\mathrm{Re}(h(z))$ is single-valued on $\mathcal{R}$ and harmonic away from the poles.  It is determined up to a real integration constant which we choose so that $\mathrm{Re}(h(\alpha))=0$.  Then $\mathrm{Re}(h(z))$ takes opposite signs on the two sheets of $\mathcal{R}$ at points corresponding to the same value of $z$ and so also $\mathrm{Re}(h(\beta))=0$ as $\alpha$ and $\beta$ are the only two points common to both sheets.

We need to determine the zero level curves of $\mathrm{Re}(h(z))$, which effectively lie on the $z$-plane by choice of integration constant.  Some of the level curves therefore coincide with the three v-trajectories\footnote{A v-trajectory of a quadratic differential $Q(z)\,\dd z^2$ is a maximal curve $\mathbb{R}\ni t\mapsto z(t)\in\mathbb{C}$, $|z'(t)|=1$, along which $Q(z(t))z'(t)^2<0$.  Such trajectories are frequently called ``vertical'' in the literature \cite{Jenkins58,Strebel84} because they are mapped by the primitive $z\mapsto w=\int^z \sqrt{Q(z')}\,\dd z'$ to vertical lines in the $w$-plane.  However these curves are rarely vertical in the native $z$-plane, so we opt for alternate terminology to avoid confusion.  Likewise maximal curves along which $Q(z)\,\dd z^2>0$ are ``horizontal'' and we will call them h-trajectories.  A trajectory is either a v-trajectory or an h-trajectory.} of the quadratic differential $Q(z)\,\dd z^2$, that by local analysis emanate from each of the branch points $z=\alpha$ and $z=\beta$ at equal angles of $\tfrac{2}{3}\pi$.  Since \emph{all} level curves of $\mathrm{Re}(h(z))$ are projections $\mathcal{R}\to\mathbb{C}$ of v-trajectories of $Q(z)\,\dd z^2$ on $\mathcal{R}$, and since this function is not identically constant on $\mathcal{R}$, there can be no ``divergent'' v-trajectories on $\mathcal{R}$, i.e., v-trajectories that are neither closed curves nor terminate in both directions at zeros or poles of $Q(z)\,\dd z^2$.  Indeed, according to \cite[Theorem 11.1, pg.\@ 48]{Strebel84}, any such v-trajectory is also ``recurrent'', and then by \cite[Corollary (1), pg.\@ 51]{Strebel84}, the limit set of the recurrent v-trajectory has a non-empty connected interior, i.e., a domain on $\mathcal{R}$ on which $\mathrm{Re}(h(z))$ is necessarily constant, yielding a contradiction as $\mathrm{Re}(h(z))$ is nonconstant and harmonic on $\mathcal{R}\setminus\{\text{poles}\}$.  Clearly any divergent v-trajectory of $Q(z)\,\dd z^2$ on $\mathbb{C}$ is the projection under $\mathcal{R}\to\mathbb{C}$ of a divergent v-trajectory on $\mathcal{R}$, so there can be no such v-trajectories in the complex $z$-plane either.  In this situation, the ``Basic Structure Theorem'' (see \cite[pg.\@ 37]{Jenkins58}) asserts that the closure $\critclosure $ of the union of the three v-trajectories emanating from each of $\alpha$ and $\beta$ together with the four v-trajectories emanating from the double zero $z=\gamma$ (the union of the so-called ``critical'' v-trajectories) has an empty interior and divides the $z$-sphere $\overline{\mathbb{C}}$ into finitely many domains, each of which is foliated by non-critical v-trajectories.  Each component of $\overline{\mathbb{C}}\setminus\critclosure $ has at least one of the critical points $z=\alpha,\beta,\gamma$ on its boundary, and on each component that has either $\alpha$ or $\beta$ on its boundary the strict inequality $|\mathrm{Re}(h(z))|>0$ holds.  Since $z=0$ is a simple pole of $h'(z)$ with real residue, the component of $\overline{\mathbb{C}}\setminus\critclosure $ containing the origin is a ``circle domain'' $\circledomain$ foliated by noncritical v-trajectories that are all Jordan curves enclosing the origin, and having at least one of $\alpha$, $\beta$, or $\gamma$ on its boundary.  Under the scaling $z=\beta Z$, we find that $Q(z)\,\dd z^2 = [(1-Z)Z^{-2}+\bo(y^{-4})]\,\dd Z^2$ where we have used $\alpha\beta\gamma^2=16$ (cf.\@ \eqref{eq:case-iv-coefficient-match}) and where the error term is uniform for bounded $Z$.  The simplified quadratic differential has only one critical point, $Z=1$ corresponding to $z=\beta$, so the latter critical point alone lies on the boundary of the circle domain containing the origin in the $z$-plane.  When $|y|$ is sufficiently large, this boundary therefore consists of a single v-trajectory that terminates at $z=\beta$ in both directions, leaving only one v-trajectory emanating from $z=\beta$ yet to be accounted for.

Suppose further that $0< \epsilon\le \arg(y)\le \tfrac{1}{2}\pi-\epsilon$.  Then for $|y|$ sufficiently large there can be no critical v-trajectory connecting $z=\gamma$ with either $z=\alpha$ or $z=\beta$.  Indeed if there were such a v-trajectory, then it would follow that $\mathrm{Re}(h(\gamma))=0$.  But a calculation shows that
\eq
\mathrm{Re}(h(\gamma))=
\mathrm{Re}\left(\int_{\alpha}^\gamma h'(z)\,\dd z\right) = 
\pm\frac{\mathrm{Re}(y)\mathrm{Im}(y)}{\sqrt{3}}+o(y^2),\quad y\to\infty,
\label{eq:genus-zero-h-of-gamma-Okamoto}
\endeq
which cannot vanish under the indicated condition on $\arg(y)$.  Therefore, the v-trajectory emanating from $z=\beta$ that does not return to $\beta$ either terminates at $z=\alpha$ or escapes to $z=\infty$.  Likewise, the three v-trajectories emanating from $z=\alpha$ either return to $z=\alpha$, terminate at $z=\beta$, or escape to $z=\infty$.  All four v-trajectories emanating from $z=\gamma$ return to $z=\gamma$ or escape to $z=\infty$.  
We may rule out the scenarios in which a v-trajectory from $z=\alpha$ or $z=\gamma$ returns to the same point by using Teichm\"uller's Lemma (see \cite[pg.\@ 71]{Strebel84}).  This is an index identity that applies to Jordan curves $\Gamma$ that are the unions of trajectories and junction points that can be poles or zeros of $Q(z)$ and that equates a left-hand side $L$ computed from data involving the orders of $Q$ at the junction points and the interior angles of $\Gamma$ at those points with a right-hand side $R$ computed from the orders of poles and zeros of $Q(z)$ in the interior of $\Gamma$.  The precise statement is the following.
\begin{lemma}[Teichm\"uller's Lemma]
Let $\Gamma$ be a Jordan curve that is the closure of the union of finitely many trajectories of a rational quadratic differential $Q(z)\,\dd z^2$, the endpoints of each of which are poles or zeros of $Q(z)$ forming the vertices of $\Gamma$.  Define indices $L$ and $R$ by
\eq
L:=\sum_{\text{vertices $j$}}\left(1-\theta_j\frac{n_j+2}{2\pi}\right)
\label{eq:Teichmueller-L}
\endeq
where $\theta_j$ is the interior angle of $\Gamma$ at the vertex and $n_j$ is the order of the rational function $Q(z)$ at the vertex (positive for zeros, negative for poles), and
\eq
R:=2+\sum_{\text{interior points $z$}} n(z)
\label{eq:Teichmueller-R}
\endeq
where $n(z)$ is the order of $Q(z)$ at a point $z$ ($n(z)=0$ if $z$ is not a zero or pole of $Q(z)$, hence the sum is finite).  Then $L=R$.
\label{lem:Teichmueller}
\end{lemma}
To apply this result in the present context, note that if $\Gamma$ is the closure of a single trajectory that terminates at the same zero of $Q(z)$ in both directions, then $L\le 0$.  Since the only pole of $Q(z)$ in the finite $z$-plane is a double pole at the origin, $R\ge 0$ with equality if and only if $\Gamma$ encloses the origin but none of the zeros of $Q(z)$.  However this equality forces the closure of $\Gamma$ to be the boundary of the circle domain which we have shown contains $z=\beta$ but not $z=\alpha$ or $z=\gamma$.  Therefore, all four v-trajectories emanating from $z=\gamma$ escape to $z=\infty$, and either (a) the remaining v-trajectory emanating from $z=\beta$ terminates at $z=\alpha$ leaving two additional v-trajectories emanating from $z=\alpha$ that must escape to $z=\infty$ or (b) the remaining v-trajectory emanating from $z=\beta$ and all three v-trajectories emanating from $z=\alpha$ escape to $z=\infty$.  Note that a v-trajectory that escapes to $z=\infty$ must approach infinity asymptotically in one of the four directions $\arg(z)=\pm\tfrac{1}{4}\pi,\pm\tfrac{3}{4}\pi$.

To determine whether (a) or (b) holds, and also to determine the manner that the four v-trajectories emanating from $z=\gamma$ tend to infinity, let us further restrict attention to the case that $\arg(y)=\tfrac{1}{4}\pi$.  Then rescaling $z$ by $z=yZ$, we find that $Q(z)\,\dd z^2 = [\tfrac{1}{16}y^4(Z+\tfrac{8}{3})(Z+\tfrac{2}{3})^2Z^{-1} + \bo(y^3)]\,\dd Z^2$, where the error term is uniform for $Z$ and $Z^{-1}$ bounded.  Since $\arg(y)=\tfrac{1}{4}\pi$ implies that $y^4<0$, we see that to leading order as $|y|\to\infty$, the v-trajectories are Schwarz symmetric in the $Z$-plane, and the rays $Z<-\tfrac{8}{3}$ and $Z>0$ are themselves v-trajectories.  In the $z$-plane, this corresponds to reflection symmetry to leading order about the diagonal $\mathrm{Re}(z)=\mathrm{Im}(z)$.  In the case that also $\kappa=0$,  this symmetry is exact
%\footnote{The symmetry is not exact however for $\kappa=\pm 3$, even though as for $\kappa=0$ the eight branch points of $\gamma$ in the $y$-plane are symmetric in reflection through the diagonal and antidiagonal.  Compare the left-hand and right-hand panels of Figure~\ref{fig:GenusZero-y-Diag-Domains}.} 
for all $|y|>0$.  For $|y|$ large then, the remaining v-trajectory emanating from $z=\beta$ escapes to $z=\infty$ in the direction $\arg(z)=\tfrac{1}{4}\pi$, so (b) holds.  The same asymptotic symmetry shows that for large $|y|$ one v-trajectory emanating from $z=\alpha$ escapes to $z=\infty$ in the direction $\arg(z)=-\tfrac{3}{4}\pi$.  In the large $|y|$ limit the two other v-trajectories emanating from $z=\alpha$ and escaping to $z=\infty$ are reflections of each other through the diagonal and are therefore confined to the two half-planes separated by that diagonal because v-trajectories cannot intersect.  The last thing to determine is the direction of escape to infinity for these two v-trajectories.  In fact, the v-trajectory lying below (resp., above) the diagonal  must escape to infinity in the direction $\arg(z)=-\tfrac{1}{4}\pi$ (resp., $\arg(z)=\tfrac{3}{4}\pi$).  Indeed, if we suppose to the contrary that the v-trajectory below the diagonal escapes in the direction $\arg(z)=-\tfrac{3}{4}\pi$, then applying Teichm\"uller's Lemma on the $z$-sphere to the curve $\Gamma$ made up of this v-trajectory and the v-trajectory emanating from $z=\alpha$ and trapped along the diagonal in the direction $\arg(z)=-\tfrac{3}{4}\pi$ with interior angles $\theta=\tfrac{2}{3}\pi$ at $\alpha$ and $\theta=0$ at $\infty$, the left-hand side is $L=1$, but as there are no poles or zeros in the interior of $\Gamma$, $R=2$, a contradiction.  Likewise, if we suppose that the v-trajectory below the diagonal escapes to infinity in the direction $\arg(z)=\tfrac{1}{4}\pi$, then again taking $\Gamma$ to consist of the same v-trajectories making an interior angle of $\theta=\tfrac{2}{3}\pi$ at $\alpha$ and $\theta=\pi$ at $\infty$, we calculate that $L=3$ while again $R=2$, again a contradiction (here we need to use the fact that in the local coordinate $k=z^{-1}$ at infinity, $Q$ has a pole of order $6$ at $k=0$).  Therefore the only remaining direction of approach to infinity for the v-trajectory emanating from $z=\alpha$ into the half-plane below the diagonal is $\arg(z)=-\tfrac{1}{4}\pi$.  By reflection through the diagonal, the v-trajectory exiting $\alpha$ into the half-plane above the diagonal escapes in the direction $\arg(z)=\tfrac{3}{4}\pi$.     This settles the behavior of the critical v-trajectories emanating from $z=\alpha,\beta$ for $\arg(y)=\tfrac{1}{4}\pi$ and $|y|$ large.  Similar analysis applied to the v-trajectories emanating from $z=\gamma$, which lies asymptotically on the diagonal between $z=\alpha$ and $z=0$ shows that one v-trajectory escapes to infinity in the direction $\arg(z)=-\tfrac{1}{4}\pi$, another escapes to infinity in the direction $\arg(z)=\tfrac{3}{4}\pi$, and the remaining two escape to infinity in the same direction, $\arg(z)=\tfrac{1}{4}\pi$, but on either side of the circle domain containing $z=0$ and having $z=\beta$ on its boundary.

It follows that there are seven components of $\overline{\mathbb{C}}\setminus\critclosure $:  
\begin{itemize}
\item
One circle domain $\circledomain$ containing $z=0$ with $z=\beta$ on its boundary and excluding $z=\alpha,\gamma$.  The inequality $|\mathrm{Re}(h(z))|>0$ holds strictly on the interior and $\mathrm{Re}(h(z))=0$ on the boundary.  The boundary $\partial\circledomain$ is a Jordan curve.
\item 
Four ``end domains'' abutting the point at infinity:
\begin{itemize}
\item one bounded by the critical v-trajectories emanating from $z=\alpha$ and escaping to infinity in the directions $\arg(z)=\tfrac{3}{4}\pi$ and $\arg(z)=-\tfrac{3}{4}\pi$ and containing the direction $\arg(-z)=0$ for large $z$, 
\item one bounded by the critical v-trajectories emanating from $z=\alpha$ and escaping to infinity in the directions $\arg(z)=-\tfrac{1}{4}\pi$ and $\arg(z)=-\tfrac{3}{4}\pi$ and containing the direction $\arg(z)=-\tfrac{1}{2}\pi$ for large $z$.
\item one bounded by the two critical v-trajectories emanating from $z=\gamma$ into the half-plane above the diagonal and containing the direction $\arg(z)=\tfrac{1}{2}\pi$ for large $z$, and
\item one bounded by the two critical v-trajectories emanating from $z=\gamma$ into the half-plane below the diagonal and containing the direction $\arg(z)=0$ for large $z$.
\end{itemize}
The former two end domains are each mapped by an analytic branch of $h(z)$ onto the open right or left half-plane, and hence $|\mathrm{Re}(h(z))|>0$ holds strictly on each while $\mathrm{Re}(h(z))=0$ on the boundary.  However, the latter two end domains are each mapped onto a horizontal translation of the left or right half-plane, since $\mathrm{Re}(h(\gamma))\neq 0$.  
\item Two ``strip domains'':
\begin{itemize}
\item
one with $z=\alpha$ and $z=\gamma$ on its boundary, foliated by v-trajectories escaping to infinity in opposite directions $\arg(z)=-\tfrac{1}{4}\pi$ and $\arg(z)=\tfrac{3}{4}\pi$, and
\item
one with $z=\gamma$ and $z=\beta$ on its boundary (the latter actually being two points of the boundary), foliated by v-trajectories escaping in both directions to $z=\infty$ in the same direction $\arg(z)=\tfrac{1}{4}\pi$ and wrapping around the circle domain.
\end{itemize}
The strip domains are each mapped by an analytic branch of $h(z)$ to a true vertical strip with the imaginary axis as one boundary.  Hence $|\mathrm{Re}(h(z))|>0$ holds on the interior of each domain and $\mathrm{Re}(h(z))=0$ holds on the part of the boundary mapped to the imaginary axis.
\end{itemize}
Therefore, the only components of $\overline{\mathbb{C}}\setminus\critclosure $ that might contain points with $\mathrm{Re}(h(z))=0$ are the two end domains with $z=\gamma$ on their boundaries.  However, 
one can see that neither of these end domains is mapped by $h(z)$ onto a half-plane containing the imaginary axis, because $z=\gamma$ is a simple saddle point of $\mathrm{Re}(h(z))$ at which these end domains dovetail with the two strip domains at a common boundary where $\mathrm{Re}(h(z))\neq 0$.  It finally follows that the zero level set of $\mathrm{Re}(h(z))$ consists precisely of the two disjoint components of $\critclosure $ that contain $z=\alpha$ and $z=\beta$ respectively.  (The remaining component containing $z=\gamma$ necessarily lies on a different level of $\mathrm{Re}(h(z))$.)  See Figure~\ref{fig:GenusZero-y-Diag-Domains}.
\begin{figure}[h]
\begin{center}
\includegraphics{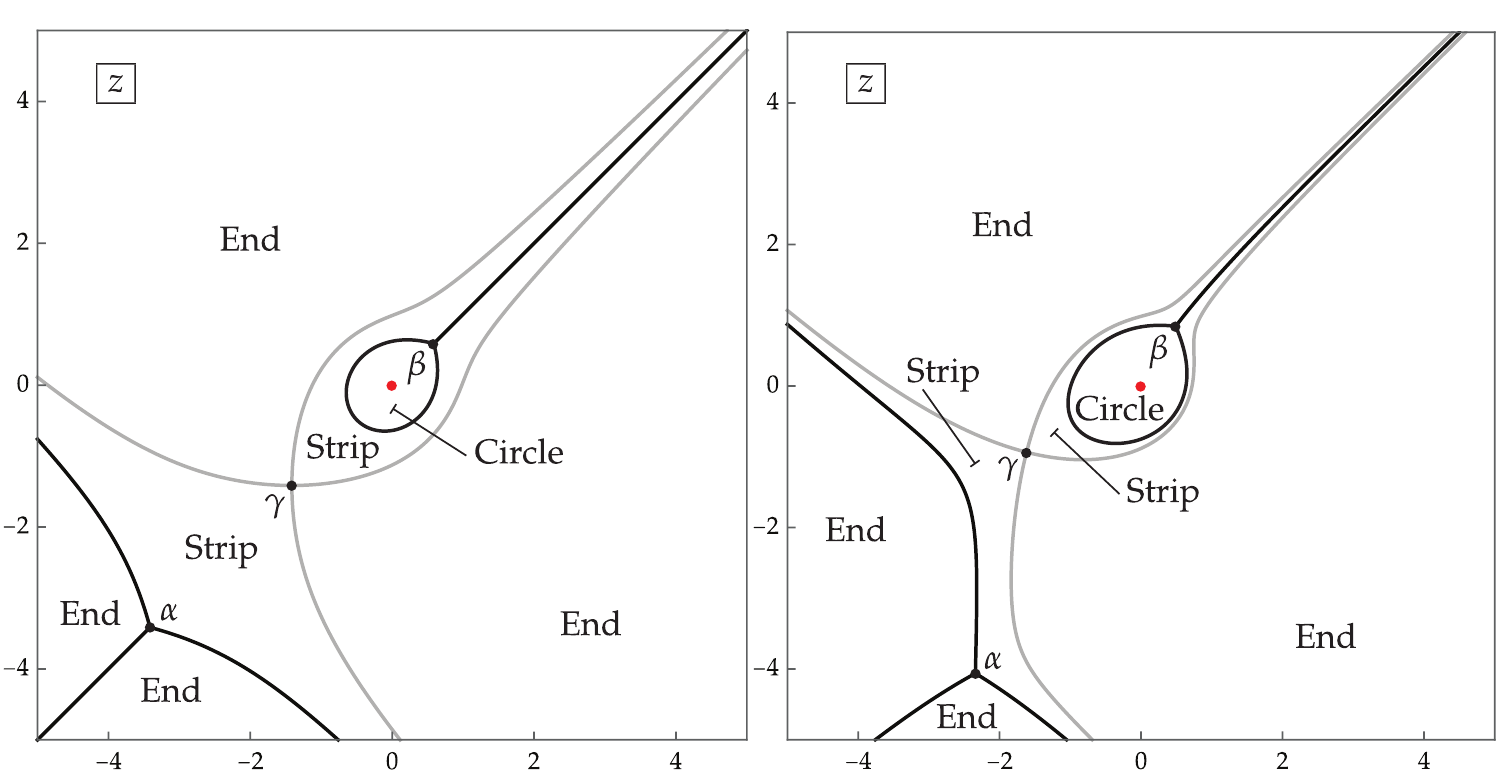}
\end{center}
\caption{The critical v-trajectories of $Q(z)\,\dd z^2$ and the seven domains into which they separate the $z$-sphere.  Left:  $\kappa=0$ and $y=2\ee^{\ii\pi/4}$ (exact symmetry through the diagonal).  Right:  $\kappa=0.5$ and $y=1.8\ee^{\ii\pi/4}$. The closure of the union of the black trajectories is the zero level set of $\mathrm{Re}(h(z))$.}
\label{fig:GenusZero-y-Diag-Domains}
\end{figure}

Having understood the global v-trajectory structure for large $|y|$ with $\arg(y)=\tfrac{1}{4}\pi$, we can analytically continue $\alpha$, $\beta$, and $\gamma$ as functions of $y$ along any path that avoids all eight branch points (roots $y$ of $B(y;\kappa)$ defined by \eqref{eq:branch-points}), and these three points will remain distinct.  Under such continuation, the global structure will remain topologically identical as long as $\mathrm{Re}(h(\gamma))=\mathrm{Re}(h^{(s,\kappa)}(\gamma;y))\neq 0$ for $\gamma=U_{0,\mathrm{gO}}(y;\kappa)$.  Note that while $\mathrm{Re}(h(z))$ is only determined up to a sign until a specific branch is selected (see below), the condition $\mathrm{Re}(h(\gamma))\neq 0$ is unambiguous.   

Now we explain how to properly define $h(z)$ as an analytic function in the $z$-plane.  First we choose branch cuts for $R(z)$ as illustrated in one of the two cases shown in Figure~\ref{fig:GenusZero-y-Diag-BranchCuts}.  
\begin{figure}[h]
\begin{center}
\includegraphics{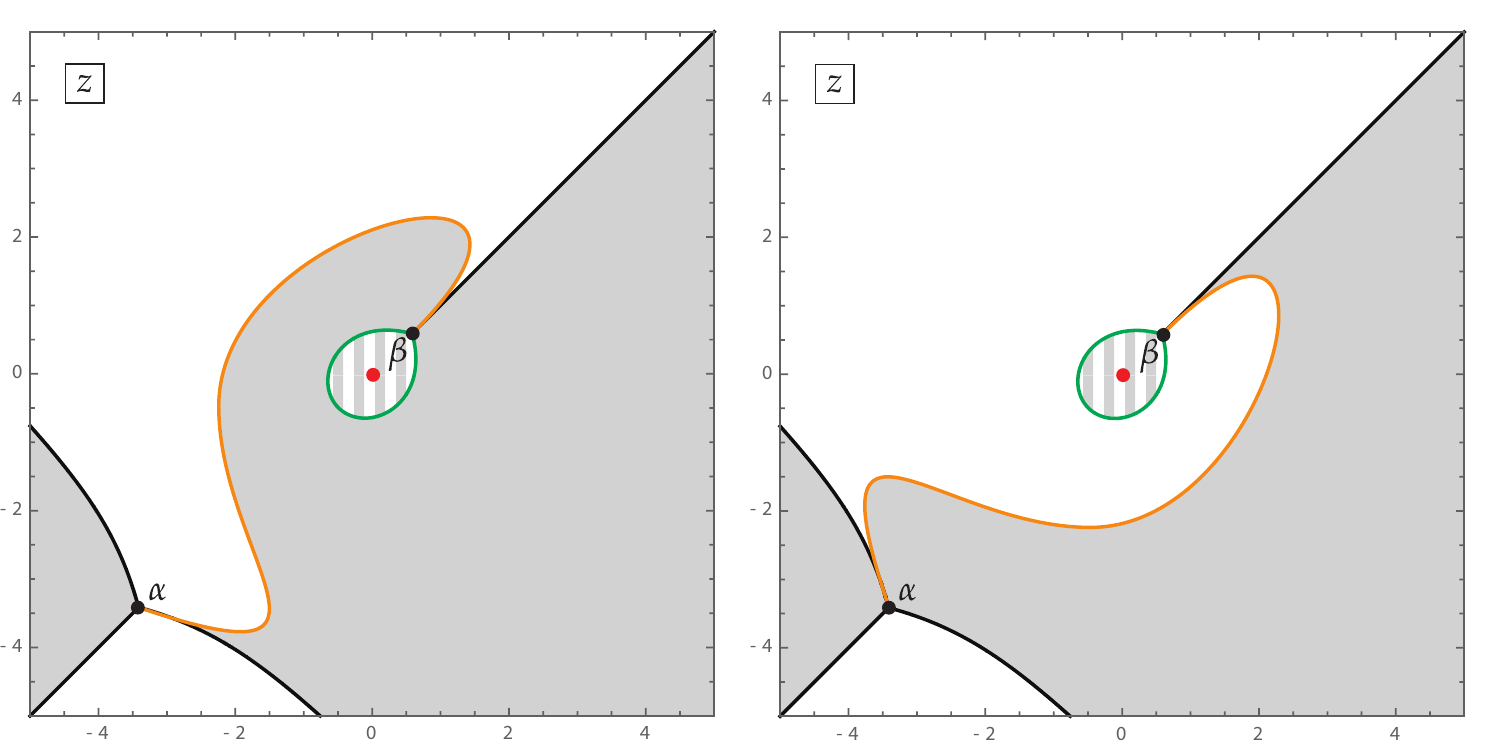}
\end{center}
\caption{Two ways to choose the branch cuts of $R(z)$ in the $z$-plane so that $h'(z)$ becomes an analytic function on the complement leading to $\mathrm{Re}(h(z))>0$ (resp., $\mathrm{Re}(h(z))<0$) in the unshaded (resp., shaded) regions, and $\mathrm{sgn}(\mathrm{Re}(h(z)))=s$ in the circle domain $\circledomain$ indicated with stripes.  In both cases the orange arc denoted $\Sigma_\mathrm{c}$ is included in the branch cut, and in the case illustrated in the left-hand (resp., right-hand) panel the 
closed component 
%$C$ 
$\partial\circledomain$
of the level curve surrounding the origin is also part of the branch cut if and only if $s=-1$ (resp., $s=+1$).  The arc $\Sigma_\mathrm{c}$ lies in the interior of the closure of the union of end and strip domains that abut $z=\gamma$, and is tangent to the v-trajectories emanating from $\alpha$ and $\beta$ as indicated.}
\label{fig:GenusZero-y-Diag-BranchCuts}
\end{figure}
Then, with $R(z)$ well defined and analytic in the complement of its branch cut, we obtain 
$h'(z)$ from \eqref{eq:h-g-phi} as a function analytic in the same domain except for a simple pole at $z=0$.  We then choose a simple curve $\Sigma_0$ in the circle domain $\circledomain$ 
%interior of $C$ 
connecting the origin with $z=\beta$ and a point 
%$z_+\in C\setminus\{\beta\}$ 
$z_+\in \partial\circledomain\setminus\{\beta\}$ 
such that for any simple arc $L$ from $z_+$ to $\beta$ via $z=0$ in the interior of $C$,
\eq
\mathrm{P.V.}\int_L h'(z)\,\dd z = 0.  
\label{eq:sum-of-BVs-on-Sigmac}
\endeq
The existence of such a point follows from the Intermediate Value Theorem.
Then: 
\begin{itemize}
\item
In the configuration shown in the left-hand panel of Figure~\ref{fig:GenusZero-y-Diag-BranchCuts}, we choose a continuation $\Sigma_{4,3}$ of $\Sigma_\mathrm{c}$ tangent to $\Sigma_\mathrm{c}$ at $z=\alpha$ and connecting $z=\alpha$ to $z=\infty$ in the 
shaded region
%sector labeled ``$-$'' 
with asymptotic angle $\arg(-z)=0$.  Then we define $\Sigma_h$ as $\Sigma_h:=\Sigma_0\cup\Sigma_\mathrm{c}\cup\Sigma_{4,3}$ if $s=+1$ or 
%$\Sigma_h:=\Sigma_0\cup\Sigma_\mathrm{c}\cup\Sigma_{4,3}\cup C$ 
$\Sigma_h:=\Sigma_0\cup\Sigma_\mathrm{c}\cup\Sigma_{4,3}\cup \partial\circledomain$ 
if $s=-1$.
\item 
In the configuration shown in the right-hand panel of Figure~\ref{fig:GenusZero-y-Diag-BranchCuts}, we choose a continuation $\Sigma_{4,1}$ of $\Sigma_\mathrm{c}$ tangent to $\Sigma_\mathrm{c}$ at $z=\alpha$ and connecting $z=\alpha$ to $z=\infty$ in the 
%sector labeled ``$+$'' 
unshaded region
with asymptotic angle $\arg(z)=-\tfrac{1}{2}\pi$.  Then we define $\Sigma_h$ as $\Sigma_h:=\Sigma_0\cup\Sigma_\mathrm{c}\cup\Sigma_{4,1}$ if $s=-1$ or  
%$\Sigma_h:=\Sigma_0\cup\Sigma_\mathrm{c}\cup\Sigma_{4,1}\cup C$ 
$\Sigma_h:=\Sigma_0\cup\Sigma_\mathrm{c}\cup\Sigma_{4,1}\cup \partial\circledomain$ 
if $s=+1$.
\end{itemize}
Finally, we define $h(z)=h^{(s,\kappa)}(z;y)$ for $z\in\mathbb{C}\setminus\overline{\Sigma_h}$ by integration of $h'(z)$ from $z=z_+$ along an arbitrary path in the same domain.
Then it follows that $\mathrm{Re}(h(z))$ is well-defined and continuous for $z\in\mathbb{C}\setminus(\{0\}\cup\Sigma_\mathrm{c})$, vanishing only along the black and green curves and elsewhere having the signs shown in Figure~\ref{fig:GenusZero-y-Diag-BranchCuts}.  It is harmonic in the same domain except for the closed curve $C$ (but only if the sign is the same in the interior and exterior; otherwise it is also harmonic on $C$).

The boundary values $h_+(z)$ and $h_-(z)$ taken by $h(z)=h^{(s,\kappa)}(z;y)$ on the arcs of its jump contour $\Sigma_h$ from the left and right, respectively, (according to the orientations shown in Figure~\ref{fig:GenusZero-y-Diag-JumpContours} below) are related as follows, recalling the notation in Section~\ref{sec:notation}.  
\begin{itemize}
\item
A residue calculation for the pole of $h'(z)$ at $z=0$ shows that
\eq
\Delta h(z)=-2\pi\ii s,\quad z\in\Sigma_0.
\label{eq:h-jump-Sigma0}
\endeq
\item 
The condition \eqref{eq:sum-of-BVs-on-Sigmac} guarantees that the sum of boundary values taken by $h(z)$ across $\Sigma_\mathrm{c}$ vanishes at the endpoint $z=\beta$.  Since $R(z)$ changes sign across $\Sigma_\mathrm{c}$ it then follows that
\eq
\langle h\rangle(z)=0,\quad z\in\Sigma_\mathrm{c}.
\label{eq:h-sum-Sigmac}
\endeq
\item
Calculating a residue of $h'(z)$ at $z=\infty$ shows that
\begin{itemize}
\item In the configuration shown in the left-hand panel of Figure~\ref{fig:GenusZero-y-Diag-BranchCuts},
\eq
\Delta h(z)=2\pi\ii\kappa,\quad z\in\Sigma_{4,3}.
\label{eq:h-jump-Sigma43}
\endeq
\item In the configuration shown in the right-hand panel of Figure~\ref{fig:GenusZero-y-Diag-BranchCuts},
\eq
\Delta h(z)=-2\pi\ii\kappa,\quad z\in\Sigma_{4,1}.
\label{eq:h-jump-Sigma41}
\endeq
\end{itemize}
\item
%If $C\subset\Sigma_h$, 
If $\partial\circledomain\subset\Sigma_h$, 
then $R(z)$ changes sign across 
%$C$ 
$\partial\circledomain$
while $h(z_+)=0$ unambiguously, so
\eq
\langle h\rangle(z)=0,\quad z\in 
%C\subset\Sigma_h.
\partial\circledomain\subset\Sigma_h.
\label{eq:h-sum-C-in-Sigmah}
\endeq
Otherwise, $h(z)$ is analytic on 
%$C\setminus\{\beta\}$, 
$\partial\circledomain\setminus\{\beta\}$, 
across which curve $\mathrm{Re}(h(z))$ changes sign.
\end{itemize}
\begin{remark}
\label{rem:TwoConfigurations}
The reason for simultaneously introducing two different configurations for the branch cut $\Sigma_\mathrm{c}$ is ultimately to be able to obtain accurate asymptotics for $y$ on both sector boundaries of the first quadrant:  $\arg(y)=0$ and $\arg(y)=\tfrac{1}{2}\pi$.  While both configurations are valid for $\arg(y)\in (0,\tfrac{1}{2}\pi)$, on each boundary ray one is forced to work with a specific choice; see Section~\ref{sec:ExteriorAxes} for details.
\end{remark}

\subsection{Use of $g(z)$ to transform $\mathbf{M}(z)$ to $\mathbf{O}(z)$}
Next, depending upon which of the two configurations of branch cuts is selected, we choose either the ``leftward'' or ``downward'' deformation of the jump contour for $\mathbf{M}(z)$ and lay the jump contour over the sign chart for $\mathrm{Re}(h(z))$ as shown in the two panels of Figure~\ref{fig:GenusZero-y-Diag-JumpContours}.  In particular, we take the closed contour $C\subset\Sigma$ to coincide with $\partial\circledomain$.
\begin{figure}[h]
\begin{center}
\includegraphics{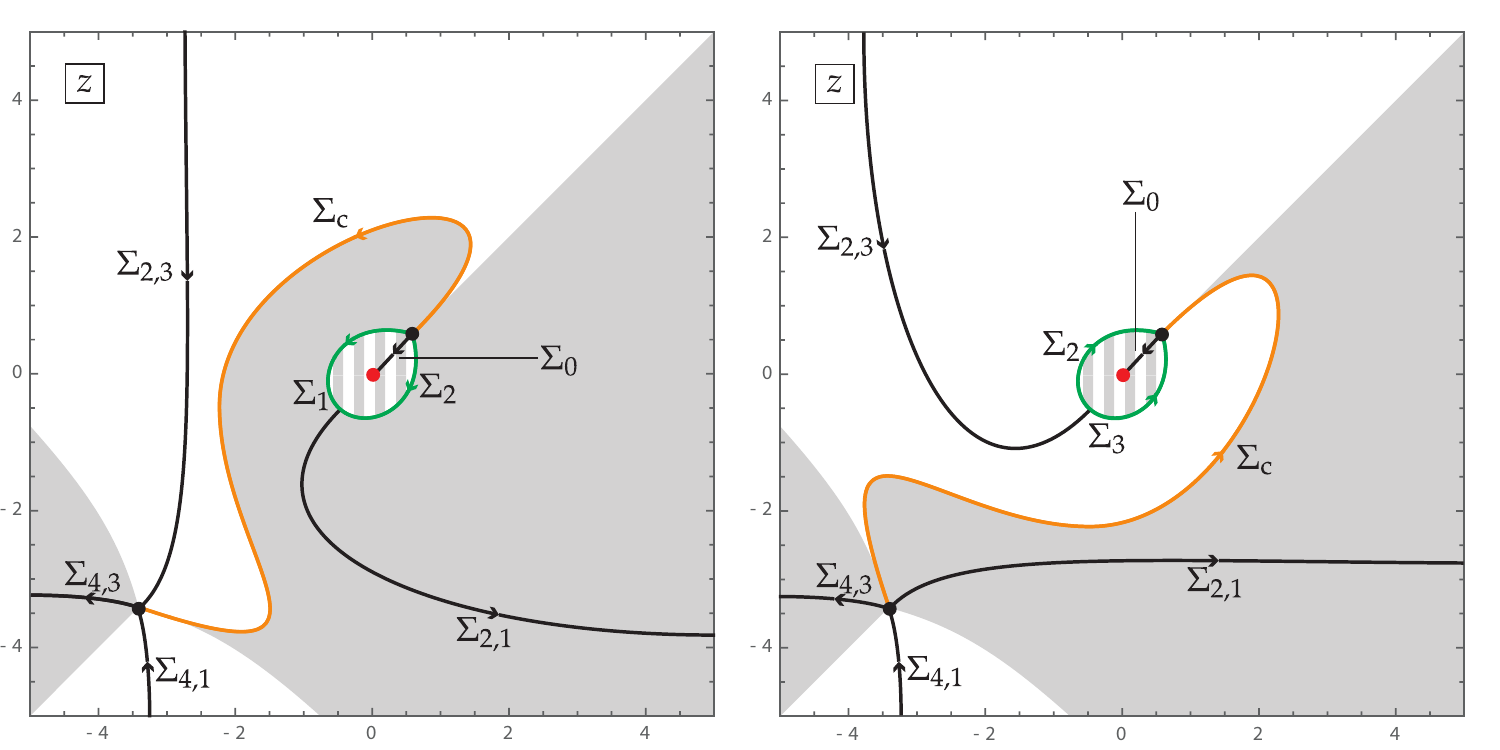}
\end{center}
\caption{Left:  the use of the ``leftward'' deformation of the jump contours for $\mathbf{M}(z)$.  Right:  the use of the ``downward'' deformation.  
The self-intersection points of $\Sigma$ lying on $C=\partial\circledomain$ are $z=\beta$ and $z=z_+$.  The sign of $\mathrm{Re}(h(z))$ is indicated by shading as in Figure~\ref{fig:GenusZero-y-Diag-BranchCuts}.}
\label{fig:GenusZero-y-Diag-JumpContours}
\end{figure}
Then, we introduce $g(z)=g^{(s,\kappa)}(z;y)=\phi(z;y)+h^{(s,\kappa)}(z;y)$ via \eqref{eq:NfromM}, in which
the constant $g_\infty$ is determined from the precise definition of $h(z)$; we will not need to know its exact value.
%
%chosen so that $g(z)=-\kappa\log(z) + g_\infty+O(z^{-1})$ as $z\to\infty$, where $\log(z)$ is cut along $\Sigma_h\setminus C$ with $-\pi<\mathrm{Im}(\log(z))<\pi$ (resp., $-\tfrac{1}{2}\pi<\mathrm{Im}(\log(z))<\tfrac{3}{2}\pi$) for large $|z|$ in the ``leftward'' (resp., ``downward'') deformation.  Therefore, we obtain the formula 
%\eq
%g_\infty=\phi(z_+;y)+\kappa\log(z_+)+\int_{z_+}^\infty\left[\phi'(z;y)+h'(z) +\frac{\kappa}{z}\right]\,\dd z,
%\endeq
%where the path of integration lies in $\mathbb{C}\setminus \Sigma_h$ and the integral 
%is absolutely convergent.  
Using \eqref{eq:h-jump-Sigma0}, it is easy to check that $\mathbf{N}(z)=\mathbf{N}^{(T,s,\kappa)}(z;y)$ has no jump discontinuity across $\Sigma_0$; then since the residue of $g'(z)$ at $z=0$ is $-s$, the condition that $\mathbf{M}(z)z^{-sT\sigma_3}$ is bounded near the origin implies that $\mathbf{N}(z)$ has a removable singularity at $z=0$ and hence is analytic in the interior of $C$.  Also $\mathbf{N}(z)\to\mathbb{I}$ as $z\to\infty$.  The matrix function $z\mapsto\mathbf{N}(z)$ is analytic except on the arcs of $\Sigma\setminus\Sigma_0$, is normalized to the identity as $z\to\infty$, and is continuous up to the jump contour in each component of its complement.  Using \eqref{eq:h-sum-Sigmac}--\eqref{eq:h-sum-C-in-Sigmah}, the jump conditions satisfied by $\mathbf{N}(z)$ take the following forms.
\eq
\mathbf{N}_+(z)=\mathbf{N}_-(z)
%\bpm 1 & -\tfrac{1}{2}\ii\ee^{-2Th(z)}\\ 0 & 1\epm,
\mathbf{U}(-\tfrac{1}{2}\ii\ee^{-2Th(z)}),
\quad z\in\Sigma_{2,3}.
\endeq
\eq
\mathbf{N}_+(z)=\mathbf{N}_-(z)
%\bpm 1 & 0\\ 2\ii\ee^{2Th(z)} & 1\epm,
\mathbf{L}(2\ii\ee^{2Th(z)}),
\quad z\in\Sigma_{2,1}.
\endeq
In the ``leftward'' configuration where $h(z)$ has a jump across $\Sigma_{4,3}$, 
\eq
\mathbf{N}_+(z)=\mathbf{N}_-(z)
%\bpm 1 & 0\\ 2\ii\ee^{-2\pi\ii T\kappa}\ee^{2T\langle h\rangle(z)} & 1\epm,
\mathbf{L}(2\ii\ee^{-2\pi\ii T\kappa}\ee^{2T\langle h\rangle(z)}),
\quad
z\in\Sigma_{4,3},\quad \text{``leftward'' configuration},
\endeq
while 
\eq
\mathbf{N}_+(z)=\mathbf{N}_-(z)
%\bpm 1 & 0\\ 2\ii\ee^{2Th(z)} & 1\epm,
\mathbf{L}(2\ii\ee^{2Th(z)}),
\quad z\in\Sigma_{4,3},\quad \text{``downward'' configuration}.
\endeq
Similarly, 
\eq
\mathbf{N}_+(z)=\mathbf{N}_-(z)
%\bpm 1 & -\tfrac{1}{2}\ii\ee^{-2Th(z)}\\ 0 & 1\epm,
\mathbf{U}(-\tfrac{1}{2}\ii\ee^{-2Th(z)}),
\quad z\in\Sigma_{4,1},\quad\text{``leftward'' configuration,}
\endeq
while as there is a jump for $h(z)$ across $\Sigma_{4,1}$ in the ``downward'' configuration,
\eq
\mathbf{N}_+(z)=\mathbf{N}_-(z)
%\bpm 1 & -\tfrac{1}{2}\ii\ee^{-2\pi\ii T\kappa}\ee^{-2T\langle h\rangle(z)}\\ 0 & 1\epm,
\mathbf{U}(-\tfrac{1}{2}\ii\ee^{-2\pi\ii T\kappa}\ee^{-2T\langle h\rangle(z)}),
\quad z\in\Sigma_{4,1},\quad\text{``downward'' configuration.}
\endeq
In both configurations $h(z)$ has a jump across $\Sigma_\mathrm{c}$, and we find that
\eq
\mathbf{N}_+(z)=\mathbf{N}_-(z)
\ee^{-2\pi\ii T\kappa}
%\bpm 0 & \tfrac{1}{2}\ii\\2\ii & 0\epm,
\mathbf{T}(2\ii),
\quad z\in\Sigma_\mathrm{c},\quad\text{``leftward'' configuration},
\label{eq:Sigma-c-left}
\endeq
while
\eq
\mathbf{N}_+(z)=\mathbf{N}_-(z)
\ee^{2\pi\ii T\kappa}
%\bpm 0 & -\tfrac{1}{2}\ii\\-2\ii & 0\epm,
\mathbf{T}(-2\ii),
\quad z\in\Sigma_\mathrm{c},\quad\text{``downward'' configuration}.
\label{eq:Sigma-c-down}
\endeq
\begin{remark}
Note that \eqref{eq:Sigma-c-left}--\eqref{eq:Sigma-c-down} are essentially the same jump condition since in both cases $\Sigma_\mathrm{c}$ is a contour connecting the same points $\alpha$ and $\beta$ but with opposite orientations in the two configurations.  Also, the jump matrices in \eqref{eq:Sigma-c-left}-\eqref{eq:Sigma-c-down} both have unit determinant, because the scalar factor satisfies $\ee^{2\pi\ii T\kappa}=\ee^{-2\pi\ii T\kappa}$ since $\Theta_\infty=-T\kappa$ and $2\Theta_\infty\in\mathbb{Z}$ holds for $(\Theta_0,\Theta_\infty)\in\Lambda_\mathrm{gO}$.   
\label{rem:same-jump}
\end{remark}
On the arcs of $C$ it appears at first that one should get different jump conditions for $\mathbf{N}(z)$ depending on whether or not $C\subset\Sigma_h$, but if one uses the fact that $\Delta g(z)=\Delta h(z)$ and observes the dichotomy that either $\Delta h(z)=0$ or $\langle h\rangle(z)=0$, one can write the jump conditions in the same form for both cases:
\eq
\mathbf{N}_+(z)=\mathbf{N}_-(z)
\ee^{-Th_-(z)\sigma_3}\mathbf{V}_k\ee^{Th_+(z)\sigma_3},\quad z\in\Sigma_k,\quad k=1,2,3,
\endeq
where $\mathbf{V}_k$, $k=1,2,3$ are the gO connection matrices defined in \eqref{eq:gO-connection}.

To deal with the jump conditions on the arcs of $C$, we use the following factorizations, all of which are special cases of \eqref{eq:general-factorizations} with two of the three factors combined:
\eq
%\begin{split}
%\mathbf{V}_1 &= \bpm 1 & 0\\ 2\ee^{\ii\pi/6} & 1\epm\bpm\tfrac{1}{\sqrt{3}}\ee^{-\ii\pi/6} & -\tfrac{1}{2}\ee^{\ii\pi/6}\\0 & \sqrt{3}\ee^{\ii\pi/6}\epm\\
%&=\bpm 1 & 0\\-2\ee^{-\ii\pi/6} & 1\epm\bpm\tfrac{1}{\sqrt{3}}\ee^{-\ii\pi/6} & -\tfrac{1}{2}\ee^{\ii\pi/6}\\2\ee^{-\ii\pi/6} & 0\epm,
%\end{split}
\begin{split}
\mathbf{V}_1 &= 
%\bpm 1 & 0\\ 2\ee^{\frac{\ii\pi}{6}} & 1\epm
\mathbf{L}(2\ee^{\frac{\ii\pi}{6}})
\bpm\tfrac{1}{\sqrt{3}} & -\tfrac{1}{2}\\0 & \sqrt{3}\epm\quad\text{(``L(DU)'')}\\
&=
%\bpm 1 & 0\\-2\ee^{-\frac{\ii\pi}{6}} & 1\epm
\mathbf{L}(-2\ee^{-\frac{\ii\pi}{6}})
\bpm\tfrac{1}{\sqrt{3}} & -\tfrac{1}{2}\\2 & 0\epm
\quad\text{(``L(TL)'')},
\end{split}
\label{eq:V1}
\endeq
\eq
%\begin{split}
%\mathbf{V}_2 &=\bpm \sqrt{3}\ee^{\ii\pi/6} & \tfrac{1}{2}\ee^{\ii\pi/6}\\0 & \tfrac{1}{\sqrt{3}}\ee^{-\ii\pi/6}\epm\bpm 1 & 0\\-2\ee^{-\ii\pi/6} & 1\epm\\
%&=\bpm 0 & \tfrac{1}{2}\ee^{\ii\pi/6}\\-2\ee^{-\ii\pi/6} & \tfrac{1}{\sqrt{3}}\ee^{-\ii\pi/6}\epm
%\bpm 1 & 0\\ 2\ee^{\ii\pi/6} & 1\epm\\
%&=\bpm \ee^{\ii\pi/3} & 0\\-\tfrac{2}{\sqrt{3}}\ee^{-\ii\pi/3} &\ee^{-\ii\pi/3}\epm
%\bpm 1 & \tfrac{1}{2}\ee^{-\ii\pi/6}\\0 & 1\epm\\
%&=\bpm \ee^{\ii\pi/3} & \tfrac{\sqrt{3}}{2}\ee^{\ii\pi/3}\\-\tfrac{2}{\sqrt{3}}\ee^{-\ii\pi/3} & 0\epm\bpm 1 & -\tfrac{1}{2}\ee^{\ii\pi/6}\\0 & 1\epm,
%\end{split}
\begin{split}
\mathbf{V}_2 &=\bpm \sqrt{3} & \tfrac{1}{2}\\0 & \tfrac{1}{\sqrt{3}}\epm
%\bpm 1 & 0\\-2\ee^{-\frac{\ii\pi}{6}} & 1\epm
\mathbf{L}(-2\ee^{-\frac{\ii\pi}{6}})
\quad\text{(``(UD)L'')}\\
&=\bpm 0 & \tfrac{1}{2}\\-2 & \tfrac{1}{\sqrt{3}}\epm
%\bpm 1 & 0\\ 2\ee^{\frac{\ii\pi}{6}} & 1\epm
\mathbf{L}(2\ee^{\frac{\ii\pi}{6}})
\quad\text{(``(LT)L'')}\\
&=\bpm \ee^{\frac{\ii\pi}{6}} & 0\\-\tfrac{2}{\sqrt{3}}\ee^{-\frac{\ii\pi}{6}} &\ee^{-\frac{\ii\pi}{6}}\epm
%\bpm 1 & \tfrac{1}{2}\ee^{-\frac{\ii\pi}{6}}\\0 & 1\epm
\mathbf{U}(\tfrac{1}{2}\ee^{-\frac{\ii\pi}{6}})
\quad\text{(``(LD)U'')}\\
&=\bpm \ee^{\frac{\ii\pi}{6}} & \tfrac{\sqrt{3}}{2}\ee^{\frac{\ii\pi}{6}}\\-\tfrac{2}{\sqrt{3}}\ee^{-\frac{\ii\pi}{6}} & 0\epm
%\bpm 1 & -\tfrac{1}{2}\ee^{\frac{\ii\pi}{6}}\\0 & 1\epm
\mathbf{U}(-\tfrac{1}{2}\ee^{\frac{\ii\pi}{6}})
\quad
\text{(``(UT)U'')},
\end{split}
\label{eq:V2}
\endeq
\eq
%\begin{split}
%\mathbf{V}_3&=\bpm 1 & \tfrac{1}{2}\ee^{-5\pi\ii/6}\\0 & 1\epm
%\bpm\ee^{-\ii\pi/3} & 0\\\tfrac{2}{\sqrt{3}}\ee^{-\ii\pi/3} & \ee^{\ii\pi/3}\epm\\
%&=\bpm 1 & \tfrac{1}{2}\ee^{-\ii\pi/6}\\0 & 1\epm
%\bpm 0 & -\tfrac{\sqrt{3}}{2}\ee^{\ii\pi/3}\\\tfrac{2}{\sqrt{3}}\ee^{-\ii\pi/3} & \ee^{\ii\pi/3}\epm.
%\end{split}
\begin{split}
\mathbf{V}_3&=
%\bpm 1 & \tfrac{1}{2}\ee^{-\frac{5\ii\pi}{6}}\\0 & 1\epm
\mathbf{U}(\tfrac{1}{2}\ee^{-\frac{5\ii\pi}{6}})
\bpm\ee^{-\frac{\ii\pi}{6}} & 0\\\tfrac{2}{\sqrt{3}}\ee^{-\frac{\ii\pi}{6}} & \ee^{\frac{\ii\pi}{6}}\epm
\quad\text{(``U(DL)'')}\\
&=
%\bpm 1 & \tfrac{1}{2}\ee^{-\frac{\ii\pi}{6}}\\0 & 1\epm
\mathbf{U}(\tfrac{1}{2}\ee^{-\frac{\ii\pi}{6}})
\bpm 0 & -\tfrac{\sqrt{3}}{2}\ee^{\frac{\ii\pi}{6}}\\\tfrac{2}{\sqrt{3}}\ee^{-\frac{\ii\pi}{6}} & \ee^{\frac{\ii\pi}{6}}\epm\quad\text{(``U(TU)'')}.
\end{split}
\label{eq:V3}
\endeq
Based on these, we transform $\mathbf{N}(z)$ into $\mathbf{O}(z)=\mathbf{O}^{(T,s,\kappa)}(z;y)$ by making the following explicit substitutions in the  ``lens'' domains $\Lambda_1$ and $\Lambda_2$ (in case we are using the ``leftward'' configuration) or $\Lambda_2$ and $\Lambda_3$ (in case we are using the ``downward'' configuration) shown in Figure~\ref{fig:GenusZero-y-Diag-N-to-O} as well as in the interior of $C$. 
\begin{figure}[h]
\begin{center}
\includegraphics{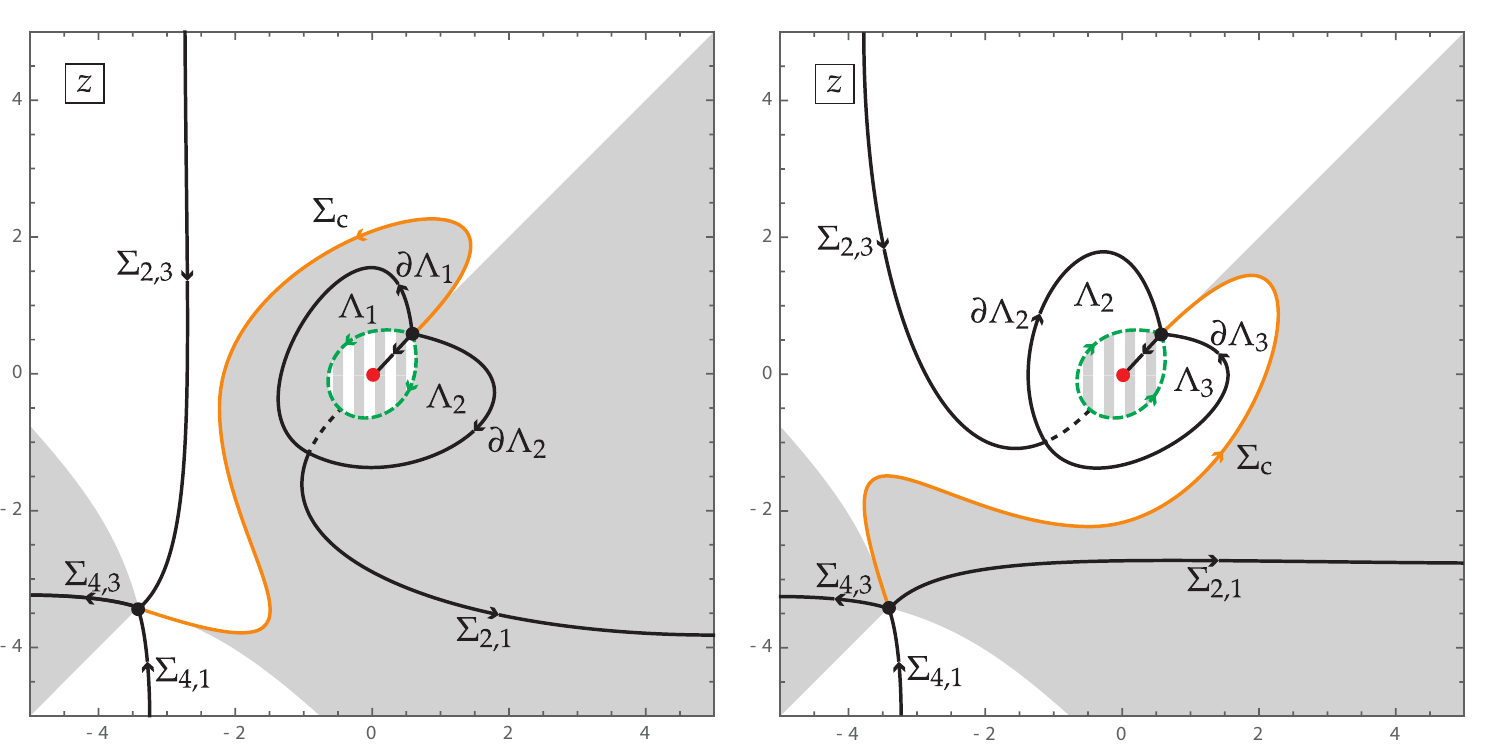}
\end{center}
\caption{The ``lens'' domains $\Lambda_j$, $j=1,2,3$ and their boundaries in the ``leftward'' configuration (left) and the ``downward'' configuration (right).  The indicated solid black and orange arcs form the final jump contour for $\mathbf{O}(z)$ (the dashed arcs have been removed by the transformation $\mathbf{N}(z)\mapsto\mathbf{O}(z)$).  Shading indicates the sign of $\mathrm{Re}(h(z))$ as in Figures~\ref{fig:GenusZero-y-Diag-BranchCuts} and \ref{fig:GenusZero-y-Diag-JumpContours}.}
\label{fig:GenusZero-y-Diag-N-to-O}
\end{figure}
\begin{itemize}
\item For the ``leftward'' configuration:
\begin{itemize}
\item If $C\not\subset\Sigma_h$ (i.e., $s=+1$), we combine the factorizations on the first lines of \eqref{eq:V1} and \eqref{eq:V2} to set
\eq
\mathbf{O}(z):=\mathbf{N}(z)
%\bpm 1 & 0\\ 2\ee^{\frac{\ii\pi}{6}}\ee^{2Th(z)} & 1\epm,
\mathbf{L}(2\ee^{\frac{\ii\pi}{6}}\ee^{2Th(z)}),
\quad z\in\Lambda_1,
\endeq
\eq
\mathbf{O}(z):=\mathbf{N}(z)
%\bpm 1&0\\ 2\ee^{-\frac{\ii\pi}{6}}\ee^{2Th(z)} & 1\epm,
\mathbf{L}(2\ee^{-\frac{\ii\pi}{6}}\ee^{2Th(z)}),
\quad z\in\Lambda_2,
\endeq
\eq
%\mathbf{O}^{(T,s,\kappa)}(z;y):=\mathbf{N}^{(T,s,\kappa)}(z;y)\bpm
%\sqrt{3}\ee^{\ii\pi/6} & \tfrac{1}{2}\ee^{\ii\pi/6}\ee^{-2Th^{(s,\kappa)}(z;y)}\\0 & \tfrac{1}{\sqrt{3}}\ee^{-\ii\pi/6}\epm,\quad \text{$z$ in the interior of $C$.}
\mathbf{O}(z):=\mathbf{N}(z)\bpm
\sqrt{3} & \tfrac{1}{2}\ee^{-2Th(z)}\\0 & \tfrac{1}{\sqrt{3}}\epm,\quad \text{$z$ in the interior of $C$.}
\label{eq:genus-zero-O-from-N-insideC-1}
\endeq
\item If $C\subset\Sigma_h$ (i.e., $s=-1$), we combine the factorizations on the second lines of \eqref{eq:V1} and \eqref{eq:V2} to set
\eq
\mathbf{O}(z):=\mathbf{N}(z)
%\bpm 1 & 0\\-2\ee^{-\frac{\ii\pi}{6}}\ee^{2Th(z)} & 1\epm,
\mathbf{L}(-2\ee^{-\frac{\ii\pi}{6}}\ee^{2Th(z)}),
\quad z\in\Lambda_1,
\endeq
\eq
\mathbf{O}(z):=\mathbf{N}(z)
%\bpm 1 & 0\\-2\ee^{\frac{\ii\pi}{6}}\ee^{2Th(z)} & 1\epm,
\mathbf{L}(-2\ee^{\frac{\ii\pi}{6}}\ee^{2Th(z)}),
\quad z\in\Lambda_2,
\endeq
\eq
%\mathbf{O}^{(T,s,\kappa)}(z;y):=\mathbf{N}^{(T,s,\kappa)}(z;y)\bpm
%0 & \tfrac{1}{2}\ee^{\ii\pi/6}\\-2\ee^{-\ii\pi/6} & \tfrac{1}{\sqrt{3}}\ee^{-\ii\pi/6}\ee^{2Th^{(s,\kappa)}(z;y)}\epm,\quad\text{$z$ in the interior of $C$.}
\mathbf{O}(z):=\mathbf{N}(z)\bpm
0 & \tfrac{1}{2}\\-2 & \tfrac{1}{\sqrt{3}}\ee^{2Th(z)}\epm,\quad\text{$z$ in the interior of $C$.}
\label{eq:genus-zero-O-from-N-insideC-2}
\endeq
\end{itemize}
\item For the ``downward'' configuration:
\begin{itemize}
\item If $C\not\subset\Sigma_h$ (i.e., $s=-1$), we combine the factorizations on the third line of \eqref{eq:V2} and the first line of \eqref{eq:V3} to set
\eq
\mathbf{O}(z):=\mathbf{N}(z)
%\bpm 1 & -\tfrac{1}{2}\ee^{-\frac{\ii\pi}{6}}\ee^{-2Th(z)} \\ 0 & 1\epm,
\mathbf{U}(-\tfrac{1}{2}\ee^{-\frac{\ii\pi}{6}}\ee^{-2Th(z)}),
\quad
z\in\Lambda_2,
\endeq
\eq
\mathbf{O}(z):=\mathbf{N}(z)
%\bpm 1 & \tfrac{1}{2}\ee^{-\frac{5\ii\pi}{6}}\ee^{-2Th(z)}\\0 & 1\epm,
\mathbf{U}(\tfrac{1}{2}\ee^{-\frac{5\ii\pi}{6}}\ee^{-2Th(z)}),
\quad
z\in\Lambda_3,
\endeq
\eq
%\mathbf{O}^{(T,s,\kappa)}(z;y):=\mathbf{N}^{(T,s,\kappa)}(z;y)\bpm
%\ee^{\ii\pi/3} & 0\\ -\tfrac{2}{\sqrt{3}}\ee^{-\ii\pi/3}\ee^{2Th^{(s,\kappa)}(z;y)} & \ee^{-\ii\pi/3}\epm,\quad\text{$z$ in the interior of $C$.}
\mathbf{O}(z):=\mathbf{N}(z)\bpm
\ee^{\frac{\ii\pi}{6}} & 0\\ -\tfrac{2}{\sqrt{3}}\ee^{-\frac{\ii\pi}{6}}\ee^{2Th(z)} & \ee^{-\frac{\ii\pi}{6}}\epm,\quad\text{$z$ in the interior of $C$.}
\label{eq:genus-zero-O-from-N-insideC-3}
\endeq
\item If $C\subset\Sigma_h$ (i.e., $s=+1$), we combine the factorizations on the fourth line of \eqref{eq:V2} and the second line of \eqref{eq:V3} to set
\eq
\mathbf{O}(z):=\mathbf{N}(z)
%\bpm 1 & \tfrac{1}{2}\ee^{\frac{\ii\pi}{6}}\ee^{-2Th(z)} \\0 & 1\epm,
\mathbf{U}(\tfrac{1}{2}\ee^{\frac{\ii\pi}{6}}\ee^{-2Th(z)}),
\quad
z\in\Lambda_2,
\endeq
\eq
\mathbf{O}(z):=\mathbf{N}(z)
%\bpm 1 & \tfrac{1}{2}\ee^{-\frac{\ii\pi}{6}}\ee^{-2Th(z)}\\0 & 1\epm,
\mathbf{U}(\tfrac{1}{2}\ee^{-\frac{\ii\pi}{6}}\ee^{-2Th(z)}),
\quad z\in\Lambda_3,
\endeq
\eq
%\mathbf{O}^{(T,s,\kappa)}(z;y):=\mathbf{N}^{(T,s,\kappa)}(z;y)\bpm
%\ee^{\ii\pi/3}\ee^{-2Th^{(s,\kappa)}(z;y)} & \tfrac{\sqrt{3}}{2}\ee^{\ii\pi/3}\\
%-\tfrac{2}{\sqrt{3}}\ee^{-\ii\pi/3} & 0\epm,\quad\text{$z$ in the interior of $C$.}
\mathbf{O}(z):=\mathbf{N}(z)\bpm
\ee^{\frac{\ii\pi}{6}}\ee^{-2Th(z)} & \tfrac{\sqrt{3}}{2}\ee^{\frac{\ii\pi}{6}}\\
-\tfrac{2}{\sqrt{3}}\ee^{-\frac{\ii\pi}{6}} & 0\epm,\quad\text{$z$ in the interior of $C$.}
\label{eq:genus-zero-O-from-N-insideC-4}
\endeq
\end{itemize}
\end{itemize}
In all cases, elsewhere that $\mathbf{N}(z)$ has a definite value we set
$\mathbf{O}(z):=\mathbf{N}(z)$.
Then $\mathbf{O}(z)$ no longer has any jump discontinuity across the arcs of $C$, and the jump across the arc of $\Sigma_{2,1}$ common to the boundary of the lens domains $\Lambda_1$ and $\Lambda_2$ is also removed in the ``leftward'' configuration, while the jump across the arc of $\Sigma_{2,3}$ common to the boundary of the lens domains $\Lambda_2$ and $\Lambda_3$ is also removed in the ``downward'' configuration.  The domain of analyticity of $\mathbf{O}(z)=\mathbf{O}^{(T,s,\kappa)}(z;y)$ is the complement of the jump contour shown with solid black and orange curves in the two panels of Figure~\ref{fig:GenusZero-y-Diag-N-to-O}, and the jump conditions satisfied by $\mathbf{O}(z)$ on the arcs $\Sigma_{2,3}$, $\Sigma_{4,3}$, $\Sigma_{4,1}$, $\Sigma_\mathrm{c}$, and the rest of $\Sigma_{2,1}$ (resp., $\Sigma_{2,1}$, $\Sigma_{4,3}$, $\Sigma_{4,1}$, $\Sigma_\mathrm{c}$, and the rest of $\Sigma_{2,3}$) for the ``leftward'' (resp., ``downward'') configuration are exactly the same as for $\mathbf{N}(z)$.  New jump conditions for $\mathbf{O}(z)$ appear on the lens boundaries and on $\Sigma_0$:
\begin{itemize}
\item
For the ``leftward'' configuration, we have
\eq
\mathbf{O}_+(z)=\mathbf{O}_-(z)
%\bpm 1 & 0\\ 2s\ee^{\frac{s\ii\pi}{6}}\ee^{2Th(z)} & 1\epm,
\mathbf{L}(2s\ee^{\frac{s\ii\pi}{6}}\ee^{2Th(z)}),
\quad z\in\partial\Lambda_1,
\endeq
\eq
\mathbf{O}_+(z)=\mathbf{O}_-(z)
%\bpm 1 & 0\\ -2s\ee^{-\frac{s\ii\pi}{6}}\ee^{2Th(z)} & 1\epm,
\mathbf{L}(-2s\ee^{-\frac{s\ii\pi}{6}}\ee^{2Th(z)}),
\quad z\in\partial\Lambda_2,
\endeq
and
\eq
\begin{split}
\mathbf{O}_+(z)&=\mathbf{O}_-(z)
%\bpm 1 & \tfrac{1}{2\sqrt{3}}s(\ee^{-2Tsh_+(z)}-\ee^{-2Tsh_-(z)})\\0 & 1\epm
\mathbf{U}(\tfrac{1}{2\sqrt{3}}s(\ee^{-2Tsh_+(z)}-\ee^{-2Tsh_-(z)})),
\\
&=\mathbf{O}_-(z)
%\bpm 1 & -\tfrac{1}{2}\ii\ee^{-2\pi\ii T\kappa}\ee^{-2sT\langle h\rangle(z)}\\ 0 & 1\epm,
\mathbf{U}(-\tfrac{1}{2}\ii\ee^{-2\pi\ii T\kappa}\ee^{-2sT\langle h\rangle(z)}),
\quad z\in\Sigma_0.
\end{split}
\label{eq:O-Sigma0-leftward}
\endeq
\item
For the ``downward'' configuration, we have
\eq
\mathbf{O}_+(z)=\mathbf{O}_-(z)
%\bpm 1 & -\tfrac{1}{2}s\ee^{\frac{s\ii\pi}{6}}\ee^{-2Th(z)}\\0 & 1\epm,
\mathbf{U}(-\tfrac{1}{2}s\ee^{\frac{s\ii\pi}{6}}\ee^{-2Th(z)}),
\quad z\in\partial\Lambda_2,
\endeq
\eq
\mathbf{O}_+(z)=\mathbf{O}_-(z)
%\bpm 1 & \tfrac{1}{2}s\ee^{-\frac{s\ii\pi}{6}}\ee^{-2Th(z)} \\0 & 1\epm,
\mathbf{U}(\tfrac{1}{2}s\ee^{-\frac{s\ii\pi}{6}}\ee^{-2Th(z)}),
\quad
z\in\partial\Lambda_3,
\endeq
and
\eq
\begin{split}
\mathbf{O}_+(z)&=\mathbf{O}_-(z)
%\bpm 1 & 0\\ \tfrac{2}{\sqrt{3}}s(\ee^{-2Tsh_+(z)}-\ee^{-2Tsh_-(z)}) & 1\epm
\mathbf{L}(\tfrac{2}{\sqrt{3}}s(\ee^{-2Tsh_+(z)}-\ee^{-2Tsh_-(z)}))
\\
&=\mathbf{O}_-(z)
%\bpm 1 &0\\-2\ii\ee^{-2\pi\ii T\kappa}\ee^{-2sT\langle h\rangle(z)} & 1\epm,
\mathbf{L}(-2\ii\ee^{-2\pi\ii T\kappa}\ee^{-2sT\langle h\rangle(z)}),
\quad z\in\Sigma_0.
\end{split}
\label{eq:O-Sigma0-downward}
\endeq
\end{itemize}
To go from the first to the second line in \eqref{eq:O-Sigma0-leftward} and \eqref{eq:O-Sigma0-downward},
we use the identity $s(\ee^{-2Tsh_+(z)}-\ee^{-2Tsh_-(z)})=-\ii\sqrt{3}\ee^{-2\pi\ii T\kappa}\ee^{-2Ts\langle h\rangle(z)}$ which follows from $h_\pm(z)=\langle h\rangle(z)\pm\tfrac{1}{2}\Delta h(z)$, the jump condition \eqref{eq:h-jump-Sigma0}, the parametrization $(\Theta_0,\Theta_\infty)=(sT,-\kappa T)$, and the gO lattice conditions \eqref{eq:gO-lattice}.
As $\mathbf{O}(z)=\mathbf{N}(z)$ for large $|z|$, the normalization condition reads $\mathbf{O}(z)\to\mathbb{I}$ as $z\to\infty$.  

\subsection{Parametrix construction}
\label{sec:Genus-Zero-Parametrix}
Comparing the jump contours shown in Figure~\ref{fig:GenusZero-y-Diag-N-to-O} with the sign charts for $\mathrm{Re}(h(z))$ shown in Figure~\ref{fig:GenusZero-y-Diag-BranchCuts}, it is now clear that the jump matrix for $\mathbf{O}(z)$ decays exponentially to $\mathbb{I}$ as $T\to +\infty$ except on the arc $\Sigma_\mathrm{c}$ and neighborhoods of its endpoints $\alpha$ and $\beta$.  To deal with these, we construct outer and inner parametrices for $\mathbf{O}(z)$.

\subsubsection{Outer parametrix}
The outer parametrix $\dot{\mathbf{O}}^{\mathrm{out}}(z)$ is defined by the properties that 
\begin{itemize}
\item it is analytic for $z\in\mathbb{C}\setminus\Sigma_\mathrm{c}$,
\item it takes continuous boundary values on $\Sigma_\mathrm{c}$ except at $z=\alpha,\beta$, where negative one-fourth power singularities are admitted, 
\item the boundary values satisfy exactly the same jump condition on $\Sigma_\mathrm{c}$ as do those of $\mathbf{O}(z)$, and
\item it tends to $\mathbb{I}$ as $z\to\infty$.  
\end{itemize}
Since the jump conditions for $\mathbf{O}(z)$ on $\Sigma_\mathrm{c}$ match those for $\mathbf{N}(z)$ given in \eqref{eq:Sigma-c-left} or \eqref{eq:Sigma-c-down}, by Remark~\ref{rem:same-jump} we may diagonalize 
%Being as the jump conditions for $\mathbf{N}(z)$ (and hence also for $\mathbf{O}(z)$) on $\Sigma_\mathrm{c}$ given by \eqref{eq:Sigma-c-left}--\eqref{eq:Sigma-c-down} are essentially identical after taking into account the orientation of $\Sigma_\mathrm{c}$ and the identity 
%$\ee^{2\pi\ii T\kappa}=\ee^{-2\pi\ii T\kappa}$ (which follows because $T\kappa=-\Theta_\infty$, and $2\Theta_\infty\in\mathbb{Z}$ whenever $(\Theta_0,\Theta_\infty)\in\Lambda_\mathrm{gO}$),
%by diagonalization of 
the constant jump matrix and hence see that the unique matrix function satisfying the above conditions can be written in the same form for both ``leftward'' and ``downward'' configurations:
\eq
\dot{\mathbf{O}}^{\mathrm{out}}(z)=\mathbf{S}_\mathrm{gO}j(z)^{\sigma_3}\mathbf{S}_\mathrm{gO}^{-1},\quad
\mathbf{S}_\mathrm{gO}:= \bpm \tfrac{1}{2}\ee^{-2\pi\ii T\kappa} & -\tfrac{1}{2}\ee^{-2\pi\ii T\kappa}\\1 & 1\epm,
%\bpm \ii & -\ii\\-(-1)^m2\ii & -(-1)^m2\ii\epm,
\label{eq:OdotOut-genus-zero}
\endeq
where $j(z)$ is the unique function analytic for $z\in\mathbb{C}\setminus\Sigma_\mathrm{c}$ with the properties that $j(z)^4=(z-\alpha)/(z-\beta)$ and $j(z)\to 1$ as $z\to\infty$.
%where the fractional power is cut along $\Sigma_\mathrm{c}$ and tends to $\mathbb{I}$ as $z\to\infty$.
By expanding $j(z)$ for large $z$, it is easy to see that 
\eq
\dot{\mathbf{O}}^{\mathrm{out}}(z)=\mathbb{I}+z^{-1}\dot{\mathbf{O}}^{\mathrm{out}}_1 + 
\bo(z^{-2}),
%z^{-2}\dot{\mathbf{O}}^{\mathrm{out}}_2+O(z^{-3}),
\quad z\to\infty, 
\label{eq:dotOout-expansion-general}
\endeq
where
\eq
\dot{O}_{1,12}^{\mathrm{out}}=\frac{1}{8}(\beta-\alpha)\ee^{-2\pi\ii T\kappa}.
\label{eq:genus-zero-outer-infty}
\endeq
%\eq
%\dot{\mathbf{O}}_1^{\mathrm{out}}:=\bpm 0 & (-1)^m\tfrac{1}{8}(\alpha-\beta)\\
%(-1)^m\tfrac{1}{2}(\alpha-\beta) & 0\epm\quad\text{and}\quad
%\dot{O}^{\mathrm{out}}_{2,12}:=(-1)^m\tfrac{1}{16}(\alpha^2-\beta^2).
%\endeq
We may also unambiguously evaluate $\dot{\mathbf{O}}^{\mathrm{out}}(z)$ at $z=0$, because the branch cut $\Sigma_\mathrm{c}$ for $j(z)$ lies in the exterior of the Jordan curve $C$ enclosing $z=0$.  From the formula \eqref{eq:OdotOut-genus-zero} we get
\eq
\begin{split}
\dot{O}^\mathrm{out}_{11}(0)\dot{O}^\mathrm{out}_{12}(0)&=\frac{1}{8}\ee^{-2\pi\ii T\kappa}\left[j(0)^2-j(0)^{-2}\right],\\
\frac{\dot{O}^\mathrm{out}_{11}(0)}{\dot{O}^\mathrm{out}_{21}(0)}&=\frac{1}{2}\ee^{-2\pi\ii T\kappa}\frac{j(0)+j(0)^{-1}}{j(0)-j(0)^{-1}}=\frac{1}{2}\ee^{-2\pi\ii T\kappa}\frac{j(0)^2+j(0)^{-2}+2}{j(0)^2-j(0)^{-2}},\\
\frac{\dot{O}^\mathrm{out}_{12}(0)}{\dot{O}^\mathrm{out}_{22}(0)}&=\frac{1}{2}\ee^{-2\pi\ii T\kappa}\frac{j(0)-j(0)^{-1}}{j(0)+j(0)^{-1}}=\frac{1}{2}\ee^{-2\pi\ii T\kappa}\frac{j(0)^2+j(0)^{-2}-2}{j(0)^2-j(0)^{-2}}.
\end{split}
\label{eq:genus-zero-outer-zero}
\endeq
To simplify the expressions \eqref{eq:genus-zero-outer-zero} further, note that using the definition of $j(z)$ and the final identity in \eqref{eq:case-iv-coefficient-match},
\eq
\left(j(0)^2\pm j(0)^{-2}\right)^2=j(0)^4 + j(0)^{-4} \pm 2 = \frac{\alpha}{\beta}+\frac{\beta}{\alpha}\pm 2 = \frac{(\alpha\pm \beta)^2}{\alpha\beta} = \frac{\gamma^2 (\alpha\pm\beta)^2}{16}.
\label{eq:GenusZero-squares-identity}
\endeq
Now from the identity $j(0)^4=\alpha/\beta$ we get the asymptotic expansion $j(0)^4 = \frac{16}{81}y^4+\bo(y^3)$ as $y\to\infty$ with $0<\arg(y)<\tfrac{1}{2}\pi$.  From the plots in Figure~\ref{fig:GenusZero-y-Diag-JumpContours} or \ref{fig:GenusZero-y-Diag-N-to-O} which show the case of $\arg(y)=\tfrac{1}{4}\pi$, we see that the correct branch of the square root to take to calculate $j(0)^2$ depends on the configuration:
\eq
j(0)^2=\begin{cases}\tfrac{4}{9}y^2 + \bo(y),\quad &\text{in the ``leftward'' configuration}\\
-\tfrac{4}{9}y^2+\bo(y),\quad &\text{in the ``downward'' configuration}
\end{cases}
\endeq
in the limit $y\to\infty$.  So, since we also have the asymptotic expansion $\tfrac{1}{4}\gamma (\alpha\pm\beta)=\tfrac{4}{9}y^2+\bo(y)$ as $y\to\infty$, we find the exact identity 
\eq
j(0)^2\pm j(0)^{-2}=\begin{cases}\tfrac{1}{4}\gamma (\alpha\pm\beta),\quad&\text{in the ``leftward'' configuration}\\
-\tfrac{1}{4}\gamma(\alpha\pm\beta),\quad&\text{in the ``downward'' configuration}.
\end{cases}
\label{eq:f(0)squared-identity-gO}
\endeq
Using this result in \eqref{eq:genus-zero-outer-zero} and combining with \eqref{eq:genus-zero-outer-infty} and the identities in \eqref{eq:case-iv-coefficient-match} gives the following combinations that will be used in Section~\ref{sec:ExteriorFormulasOkamoto} below:
\eq
\frac{\dot{O}^\mathrm{out}_{11}(0)\dot{O}^\mathrm{out}_{12}(0)}{\dot{O}^\mathrm{out}_{1,12}}=
\begin{cases}
-\tfrac{1}{4}\gamma,\quad&\text{in the ``leftward'' configuration}\\
\tfrac{1}{4}\gamma,\quad&\text{in the ``downward'' configuration};
\end{cases}
\label{eq:genus-zero-outer-u}
\endeq
\eq
\frac{\dot{O}^\mathrm{out}_{1,12}\dot{O}^\mathrm{out}_{21}(0)}{\dot{O}^\mathrm{out}_{11}(0)} = 
\begin{cases} y+\tfrac{1}{2}\gamma + 2\gamma^{-1},\quad&\text{in the ``leftward'' configuration}\\
y+\tfrac{1}{2}\gamma-2\gamma^{-1},\quad&\text{in the ``downward'' configuration};
\end{cases}
\label{eq:genus-zero-outer-ucirc-1}
\endeq
\eq
\frac{\dot{O}^\mathrm{out}_{1,12}\dot{O}^\mathrm{out}_{22}(0)}{\dot{O}^\mathrm{out}_{12}(0)} = 
\begin{cases} y+\tfrac{1}{2}\gamma - 2\gamma^{-1},\quad&\text{in the ``leftward'' configuration}\\
y+\tfrac{1}{2}\gamma+2\gamma^{-1},\quad&\text{in the ``downward'' configuration}.
\end{cases}
\label{eq:genus-zero-outer-ucirc-2}
\endeq
To finish our discussion of the outer parametrix, we observe that $\dot{\mathbf{O}}^\mathrm{out}(z)$ has unit determinant and for $z$ bounded away from $z=\alpha$ and $z=\beta$, $\dot{\mathbf{O}}^\mathrm{out}(z)$ is uniformly bounded as $T\to +\infty$.

%The parametric dependence of $\dot{\mathbf{O}}^{\mathrm{out}}(z)$ on $T$, $s$, $\kappa$, and $y$ enters explicitly via the (large) value of $m\in\mathbb{Z}$ and the dependence of $\alpha$ and $\beta$ on $y\in\mathbb{C}$ (those values of $y$ where $\mathrm{Re}(h(\gamma))=\mathrm{Re}(h^{(s,\kappa)}(\gamma;y))$ for $\gamma=\gamma^{(\kappa)}_\mathrm{gO}(y)$ remains nonzero) and $\kappa\approx\kappa_\infty$ with $\kappa_\infty\in (-1,1)$.
\subsubsection{Inner parametrices}
Inner parametrices are to be used in place of the outer parametrix in $T$-independent neighborhoods $D_\alpha$ and $D_\beta$ of $z=\alpha$ and $z=\beta$ respectively.  These are defined in terms of a $2\times 2$ matrix function $\mathbf{A}(\xi)$ satisfying the following Riemann-Hilbert conditions:
\begin{itemize}
\item
$\mathbf{A}(\xi)$ is analytic in the four sectors complementary to the four rays $\arg(\xi)=0$, $\arg(\xi)=\tfrac{2}{3}\pi$, $\arg(\xi)=-\tfrac{2}{3}\pi$, and $\arg(-\xi)=0$.
\item
$\mathbf{A}(\xi)$ takes continuous boundary values from each sector of analyticity on the union of rays forming the sector boundary.  Assuming the four rays are oriented in the direction away from the origin, the boundary values are related by the jump conditions
\eq
\mathbf{A}_+(\xi)=\mathbf{A}_-(\xi)
%\bpm 1 & 0\\ \ii\ee^{-\xi^{3/2}} & 1\epm,
\mathbf{L}(\ii\ee^{-\xi^{3/2}}),
\quad\arg(\xi)=0,
\label{eq:Airy-jump-first}
\endeq
\eq
\mathbf{A}_+(\xi)=\mathbf{A}_-(\xi)
%\bpm 1 & \ii\ee^{\xi^{3/2}}\\0 & 1\epm,
\mathbf{U}(\ii\ee^{\xi^{3/2}}),
\quad
\arg(\xi)=\pm\tfrac{2}{3}\pi,
\endeq
and
\eq
\mathbf{A}_+(\xi)=\mathbf{A}_-(\xi)
%\bpm 0 & -\ii\\-\ii&0\epm,
\mathbf{T}(-\ii),
\quad\arg(-\xi)=0.
\label{eq:Airy-jump-last}
\endeq
\item
$\mathbf{A}(\xi)$ is normalized by the condition
\eq
\mathbf{A}(\xi)\frac{1}{\sqrt{2}}\bpm 1&-1\\1&1\epm\xi^{-\sigma_3/4}=\bpm
1+\bo(\xi^{-3}) & \bo(\xi^{-1})\\\bo(\xi^{-2}) & 1+\bo(\xi^{-3})\epm,\quad\xi\to\infty.
\label{eq:Airy-Normalize}
\endeq
\end{itemize}
It is well known that there is a unique solution of these conditions, and the matrix $\mathbf{A}(\xi)$ can be explicitly written in terms of the Airy function $\mathrm{Ai}(\cdot)$ and its derivative \cite{DLMF}.  We will not need this formula, however.  The inner parametrices are defined as follows. 
Referring to the appropriate entry of Table~\ref{tab:Genus-Zero-Okamoto-Airy} corresponding to the disk $D_p$ ($p=\alpha,\beta$) and the ``leftward'' or ``downward'' configuration, we first fix the contour arc associated with $\arg(W)=0$ within $D_p$ so that a function $W(z)$ is well-defined as indicated on that arc by taking the $\tfrac{2}{3}$ power of a positive quantity.  This function can be analytically continued to the whole disk as a conformal mapping because $h'(z)^2$ has simple zeros at $z=\alpha,\beta$ and $\langle h\rangle(\alpha)=\langle h\rangle (\beta)=0$.  Once the conformal mapping $W$ is defined, we fix the remaining arcs of the jump contour within $D_p$ so that they are mapped by $W$ to the rays indicated in the table.  Noting the value of the constant matrix $\mathbf{C}$ defined in the table,
we then define a matrix $\mathbf{H}(z)$ by writing the outer parametrix in the form
\eq
\dot{\mathbf{O}}^\mathrm{out}(z)=\mathbf{H}(z)W(z)^{\frac{1}{4}\sigma_3}\frac{1}{\sqrt{2}}\bpm 1 & 1\\-1 & 1\epm\mathbf{C},\quad z\in D_p\setminus\Sigma_\mathrm{c}.
\label{eq:GenusZero-H-define}
\endeq
It is easy to see that $\mathbf{H}(z)$ has no jump across $\Sigma_\mathrm{c}\cap D_p$ and that any singularity at $W=0$ is removable.  Therefore $\mathbf{H}(z)$ is a holomorphic matrix function on $D_p$ with unit determinant, and it follows from the formula \eqref{eq:OdotOut-genus-zero} that the elements of $\mathbf{H}(z)$ are bounded on $D_p$ uniformly with respect to the large parameter $T$.  Using $\mathbf{H}(z)$ and $\mathbf{A}(\xi)$ we define an inner parametrix on $D_p$ by setting
\eq
\dot{\mathbf{O}}^{\mathrm{in},p}(z):=\mathbf{H}(z)T^{-\frac{1}{6}\sigma_3}\mathbf{A}(T^\frac{2}{3}W(z))\mathbf{C},\quad z\in D_p\setminus\{\arg(W(z))\in\{0,\pm\tfrac{2}{3}\pi,\pm\pi\}\}.
\label{eq:GenusZero-inner-parametrices-define}
\endeq
\begin{table}[h]
\caption{Data for defining the inner parametrices in $D_\alpha$ and $D_\beta$.}
\renewcommand{\arraystretch}{0.45}
\begin{tabular}{@{}|l|l|l|c|c|c|@{}}
\hline
\multirow{2}{*}{$p$}&\multirow{2}{*}{Configuration}&\multirow{2}{*}{Conformal map $W:D_p\to\mathbb{C}$}&\multicolumn{2}{c|}{Ray Preimages in $D_p$} & \multirow{2}{*}{Constant matrix $\mathbf{C}$}\\
\cline{4-5}
& &  & $\arg(W)$ & Preimage &  \\
\hline\hline
\multirow{4}{*}{$\alpha$} & \multirow{4}{*}{``leftward''} &\multirow{4}{*}{\makecell[l]{$(-2\langle h\rangle(z))^{\frac{2}{3}}$,\\ 
continued from $\Sigma_{4,3}$}} &
\shortstrut $0$ & $\Sigma_{4,3}$ & 
\multirow{4}{*}{$\ee^{-\ii\pi T\kappa\sigma_3}2^{\frac{1}{2}\sigma_3}$}  \\
\cline{4-5}
&&& \shortstrut $\tfrac{2}{3}\pi$ & $\Sigma_{4,1}$ & \\
\cline{4-5}
&&& \shortstrut $-\tfrac{2}{3}\pi$ & $\Sigma_{2,3}$ & \\
\cline{4-5}
&&& \shortstrut $\pm\pi$ & $\Sigma_\mathrm{c}$ & \\
\hline
\hline
\multirow{4}{*}{$\alpha$} & \multirow{4}{*}{``downward''} &\multirow{4}{*}{\makecell[l]{$(2\langle h\rangle(z))^{\frac{2}{3}}$,\\ 
continued from $\Sigma_{4,1}$}} &
\shortstrut $0$ & $\Sigma_{4,1}$ & 
\multirow{4}{*}{$\ii\sigma_1\ee^{\ii\pi T\kappa\sigma_3}2^{\frac{1}{2}\sigma_3}$}  \\
\cline{4-5}
&&& \shortstrut $\tfrac{2}{3}\pi$ & $\Sigma_{2,1}$ & \\
\cline{4-5}
&&& \shortstrut $-\tfrac{2}{3}\pi$ & $\Sigma_{4,3}$ & \\
\cline{4-5}
&&& \shortstrut $\pm\pi$ & $\Sigma_\mathrm{c}$ & \\
\hline
\hline
\multirow{4}{*}{$\beta$} & \multirow{4}{*}{``leftward''} &\multirow{4}{*}{\makecell[l]{$(2s\langle h\rangle(z))^{\frac{2}{3}}$,\\ 
continued from $\Sigma_0$}} &
\shortstrut $0$ & $\Sigma_0$ & 
\multirow{4}{*}{$\ii\sigma_1\ii^{\sigma_3}\ee^{\ii\pi T\kappa\sigma_3}2^{\frac{1}{2}\sigma_3}$}  \\
\cline{4-5}
&&& \shortstrut $\tfrac{2}{3}\pi$ & $\partial\Lambda_2$ & \\
\cline{4-5}
&&& \shortstrut $-\tfrac{2}{3}\pi$ & $\partial\Lambda_1$ & \\
\cline{4-5}
&&& \shortstrut $\pm\pi$ & $\Sigma_\mathrm{c}$ & \\
\hline
\hline
\multirow{4}{*}{$\beta$} & \multirow{4}{*}{``downward''} &\multirow{4}{*}{\makecell[l]{$(2s\langle h\rangle(z))^{\frac{2}{3}}$,\\ 
continued from $\Sigma_0$}} &
\shortstrut $0$ & $\Sigma_0$ & 
\multirow{4}{*}{$\ii^{\sigma_3}\ee^{-\ii\pi T\kappa\sigma_3}2^{\frac{1}{2}\sigma_3}$}  \\
\cline{4-5}
&&& \shortstrut $\tfrac{2}{3}\pi$ & $\partial\Lambda_3$ & \\
\cline{4-5}
&&& \shortstrut $-\tfrac{2}{3}\pi$ & $\partial\Lambda_2$ & \\
\cline{4-5}
&&& \shortstrut $\pm\pi$ & $\Sigma_\mathrm{c}$ & \\
\hline
\end{tabular}
\renewcommand{\arraystretch}{1}
\label{tab:Genus-Zero-Okamoto-Airy}
\end{table}
Recalling the parametrization $(\Theta_0,\Theta_\infty)=(sT,-s\kappa)$ and the gO lattice conditions \eqref{eq:gO-lattice},
in all four cases it follows from the sectorial analyticity of $\mathbf{A}(\xi)$ and the jump conditions \eqref{eq:Airy-jump-first}--\eqref{eq:Airy-jump-last} that $\dot{\mathbf{O}}^{\mathrm{in},\alpha}(z)$ (resp., $\dot{\mathbf{O}}^{\mathrm{in},\beta}(z)$) is analytic within $D_\alpha$ (resp., $D_\beta$) exactly where $\mathbf{O}(z)$ is, and satisfies exactly the same jump conditions.  
To show this it is helpful to recall the jump conditions satisfied by $h(z)$ (see \eqref{eq:h-jump-Sigma0}--\eqref{eq:h-sum-C-in-Sigmah}) and that, depending on the value of $s$, analytic continuation of $h$ across $C$ can introduce a change of sign.
%(Showing this makes use of the identities $\ee^{\pm 2\pi \ii T\kappa}=-(-1)^m$ and $\ee^{\pm\ii\pi s/6\pm 2\pi\ii T}=\pm \ii s (-1)^m$.) 
Therefore, the inner parametrices are \emph{exact local solutions} of the Riemann-Hilbert conditions characterizing $\mathbf{O}(z)$.  Moreover, it follows from \eqref{eq:Airy-Normalize} and the fact that $z\in\partial D_{\alpha,\beta}$ bounds the conformal coordinate $W(z)$ away from zero that the following estimates hold uniformly with respect to $z$ on the indicated contours:
\eq
\begin{split}
\dot{\mathbf{O}}^{\mathrm{in},\alpha}(z)\dot{\mathbf{O}}^{\mathrm{out}}(z)^{-1}&=\mathbb{I}+\bo(T^{-1}),\quad z\in\partial D_\alpha,\\
\dot{\mathbf{O}}^{\mathrm{in},\beta}(z)\dot{\mathbf{O}}^{\mathrm{out}}(z)^{-1}&=\mathbb{I}+\bo(T^{-1}),\quad z\in\partial D_\beta.
\end{split}
\label{eq:GenusZero-Mismatch}
\endeq

\subsubsection{Global parametrix and error estimation}
\label{sec:GenusZero-error}
The \emph{global parametrix} for $\mathbf{O}(z)=\mathbf{O}^{(T,s,\kappa)}(z;y)$ is then defined as follows:
\eq
\dot{\mathbf{O}}(z):=\begin{cases}\dot{\mathbf{O}}^{\mathrm{in},p}(z),&\quad
z\in D_p,\quad p\in\{\alpha,\beta\},\\
\dot{\mathbf{O}}^{\mathrm{out}}(z),&\quad 
\text{elsewhere that $\dot{\mathbf{O}}^\mathrm{out}(z)$ is analytic}.
\end{cases}
\label{eq:GenusZero-global}
\endeq
The corresponding error in modeling $\mathbf{O}(z)$ with the global parametrix is the error $\mathbf{E}(z)=\mathbf{E}^{(T,s,\kappa)}(z;y)$ defined by 
\eq
\mathbf{E}(z):=\mathbf{O}(z)\dot{\mathbf{O}}(z)^{-1}.
\label{eq:GenusZeroE}
\endeq
Since both $\mathbf{O}(z)$ and its parametrix satisfy exactly the same jump conditions within $D_\alpha$ and $D_\beta$, and also along $\Sigma_\mathrm{c}$, $\mathbf{E}(z)$ can be viewed as an analytic function in the complex $z$-plane, with only the parts lying outside of the neighborhoods $D_\alpha$ and $D_\beta$ of the black arcs of the jump contour shown in the panels of Figure~\ref{fig:GenusZero-y-Diag-N-to-O} and the boundaries $\partial D_{\alpha,\beta}$ of the neighborhoods excluded.  We take the latter to have clockwise orientation.  Since $\dot{\mathbf{O}}^{\mathrm{out}}(z)$ and its inverse are uniformly bounded for $z\in\mathbb{C}\setminus (D_\alpha\cup D_\beta)$ and $T$ sufficiently large, it is easy to check that due to the exponentially rapid convergence to the identity of the jump matrices for $\mathbf{O}(z)$ on the parts of the jump contour for $\mathbf{E}(z)$ lying outside the closure of $D_\alpha\cup D_\beta$, we have also $\mathbf{E}_+(z)=\mathbf{E}_-(z)(\mathbb{I}+\text{exponentially small})$ as $T\to+\infty$ on those arcs.  For the closed curves $\partial D_{\alpha,\beta}$ it is easy to see from \eqref{eq:GenusZero-Mismatch} that $\mathbf{E}_+(z)=\mathbf{E}_-(z)(\mathbb{I}+\bo(T^{-1}))$ holds uniformly on $\partial D_{\alpha,\beta}$ as $T\to +\infty$.  Since it also follows from \eqref{eq:GenusZeroE} that $\mathbf{E}(z)\to\mathbb{I}$ as $z\to\infty$, we may apply small norm theory in the $L^2$ setting as described, for instance, in \cite[Appendix B]{BuckinghamM:2014} to conclude that 
\eq
\mathbf{E}(z)=\mathbb{I}+z^{-1}\mathbf{E}_1+\bo(z^{-2}),\quad z\to\infty,
\label{eq:GenusZero-E-expansion}
\endeq
where $\mathbf{E}_1=\mathbf{E}_1^{(T,s,\kappa)}(y)=\bo(T^{-1})$ and that $\mathbf{E}(0)=\mathbb{I}+\bo(T^{-1})$ as $T\to +\infty$.  

\subsection{Asymptotic formul\ae\ for the gO rational solutions of Painlev\'e-IV on the exterior domain}
\label{sec:ExteriorFormulasOkamoto}
Recalling the scaling relationships between $x$ and $y$ and between $\lambda$ and $z$ given in \eqref{eq:x-y-lambda-z} we see that for $|z|$ sufficiently large, 
\eq
\begin{split}
\mathbf{Y}(\lambda;x)&=\left(\tfrac{1}{2}T^{\frac{1}{2}}\right)^{\kappa T\sigma_3}\mathbf{M}(z)\\
&=\left(\tfrac{1}{2}T^{\frac{1}{2}}\right)^{\kappa T\sigma_3}\ee^{Tg_\infty\sigma_3}\mathbf{N}(z)\ee^{-Tg(z)\sigma_3}\\
&=\left(\tfrac{1}{2}T^{\frac{1}{2}}\right)^{\kappa T\sigma_3}\ee^{Tg_\infty\sigma_3}\mathbf{O}(z)\ee^{-Tg(z)\sigma_3}\\
&=\left(\tfrac{1}{2}T^{\frac{1}{2}}\right)^{\kappa T\sigma_3}\ee^{Tg_\infty\sigma_3}\mathbf{E}(z)\dot{\mathbf{O}}(z)\ee^{-Tg(z)\sigma_3}\\
&=\left(\tfrac{1}{2}T^{\frac{1}{2}}\right)^{\kappa T\sigma_3}\ee^{Tg_\infty\sigma_3}\mathbf{E}(z)\dot{\mathbf{O}}^{\mathrm{out}}(z)\ee^{-Tg(z)\sigma_3}.
\end{split}
\label{eq:genus-zero-Y-lambda-large}
\endeq
Now using the expansions \eqref{eq:GenusZero-E-expansion},
\eqref{eq:dotOout-expansion-general}, and $g(z)=-\kappa\log(z)+g_\infty+g_1z^{-1}+\bo(z^{-2})$ as $z\to\infty$, which implies that also
\eq
\ee^{-Tg(z)\sigma_3}=\left(\mathbb{I}-Tg_1z^{-1}\sigma_3+\bo(z^{-2})\right)
\cdot\ee^{-Tg_\infty\sigma_3}z^{T\kappa\sigma_3},\quad z\to\infty,
\endeq
upon again taking into account the scaling $\lambda=\tfrac{1}{2}T^{\frac{1}{2}}z$ one sees that the matrix element $Y^\infty_{1,12}(x)$ defined in \eqref{eq:Z0-Y1-define} can be written in the form
%\eq
%\mathbf{Y}(\lambda;x)\lambda^{\Theta_\infty\sigma_3}=\mathbb{I}+\lambda^{-1}\mathbf{Y}_1(x)+O(\lambda^{-2}),\quad\lambda\to\infty,
%\endeq
%where
\eq
\begin{split}
Y^\infty_{1,12}(x)&=\left(\tfrac{1}{4}T\right)^{\kappa T}\ee^{2Tg_\infty}\tfrac{1}{2}T^{\frac{1}{2}}\left(\dot{O}^\mathrm{out}_{1,12}+E_{1,12}\right)\\
&=\left(\tfrac{1}{4}T\right)^{\kappa T}\ee^{2Tg_\infty}\tfrac{1}{2}T^{\frac{1}{2}}\left(\dot{O}^\mathrm{out}_{1,12}+\bo(T^{-1})\right),\quad T\to +\infty.
%Y^{(m,n)}_{1,22}(x)&=\tfrac{1}{2}T^{1/2}\left(E_{1,22}+Tg_1\right),\\
%Y^{(m,n)}_{2,12}(x)&=\left(\tfrac{1}{4}T\right)^{\kappa T}\ee^{2Tg_\infty}\tfrac{1}{4}T
%\left(E_{2,12}+(-1)^m\tfrac{1}{16}(\alpha^2-\beta^2)+E_{1,11}(-1)^m\tfrac{1}{8}(\alpha-\beta)\right.\\
%&\left.\quad\quad\quad{}+Tg_1E_{1,12}+Tg_1(-1)^m\tfrac{1}{8}(\alpha-\beta)\right).
\end{split}
\label{eq:genus-zero-Y112}
\endeq
Now, the first two lines of \eqref{eq:genus-zero-Y-lambda-large} are also valid for $|\lambda|$ sufficiently small, so using \eqref{eq:gprime-asymp}, the matrix $\mathbf{Y}^0_0(x)$ defined in \eqref{eq:Z0-Y1-define} can be written as
\eq
\begin{split}
\mathbf{Y}^0_0(x)&=\lim_{z\to 0}\left(\tfrac{1}{2}T^{\frac{1}{2}}\right)^{\kappa T\sigma_3}\ee^{Tg_\infty\sigma_3}\mathbf{N}(z)\ee^{-Tg(z)\sigma_3}z^{-sT\sigma_3}\left(\tfrac{1}{2}T^\frac{1}{2}\right)^{-sT\sigma_3} \\ &= \left(\tfrac{1}{2}T^{\frac{1}{2}}\right)^{\kappa T\sigma_3}\ee^{Tg_\infty\sigma_3}\mathbf{N}(0)\ee^{-Tg_0\sigma_3}\left(\tfrac{1}{2}T^\frac{1}{2}\right)^{-sT\sigma_3}.
\end{split}
\label{eq:GenusZero-Z}
\endeq
Using \eqref{eq:genus-zero-O-from-N-insideC-1}, \eqref{eq:genus-zero-O-from-N-insideC-2}, \eqref{eq:genus-zero-O-from-N-insideC-3}, and \eqref{eq:genus-zero-O-from-N-insideC-4} in which the exponential factor vanishes at the origin $z=0$ in each case, we have
\eq
\mathbf{N}(0)=\begin{cases}
\mathbf{O}(0)\mathbf{D}(\tfrac{1}{\sqrt{3}}),&\quad\text{in the ``leftward'' configuration, for $C\not\subset\Sigma_h$ (i.e., $s=1$)}\\
\mathbf{O}(0)\mathbf{T}(2),&\quad\text{in the ``leftward'' configuration, for $C\subset\Sigma_h$ (i.e., $s=-1$)}\\
\mathbf{O}(0)\mathbf{D}(\ee^{-\frac{\ii\pi}{6}}),&\quad\text{in the ``downward'' configuration, for $C\not\subset\Sigma_h$ (i.e., $s=-1$)}\\
\mathbf{O}(0)\mathbf{T}(\tfrac{2}{\sqrt{3}}\ee^{-\frac{\ii\pi}{6}}),&\quad\text{in the ``downward'' configuration, for $C\subset\Sigma_h$ (i.e., $s=1$),}
\end{cases} 
\label{eq:GenusZero-N(0)}
\endeq
where we are using the compact notation $\mathbf{D}(a)$ and $\mathbf{T}(a)$ for diagonal and off-diagonal matrix factors defined in \eqref{eq:matrix-factors-notation}.  

Then since $\mathbf{O}(0)=\mathbf{E}(0)\dot{\mathbf{O}}(0)=\mathbf{E}(0)\dot{\mathbf{O}}^\mathrm{out}(0) = (\mathbb{I}+\bo(T^{-1}))\dot{\mathbf{O}}^\mathrm{out}(0)$ and $\dot{\mathbf{O}}^\mathrm{out}(0)$ is bounded,
\eq
Y^0_{0,11}(x)Y^0_{0,12}(x)=\begin{cases}
\left(\tfrac{1}{4}T\right)^{\kappa T}\ee^{2Tg_\infty}
(\dot{O}^\mathrm{out}_{11}(0)\dot{O}^\mathrm{out}_{12}(0)+\bo(T^{-1})),&\quad C\not\subset\Sigma_h\\
-\left(\tfrac{1}{4}T\right)^{\kappa T}\ee^{2Tg_\infty}
(\dot{O}^\mathrm{out}_{11}(0)\dot{O}^\mathrm{out}_{12}(0)+\bo(T^{-1})),&\quad C\subset\Sigma_h.
\end{cases}
\label{eq:genus-zero-product-Z}
\endeq
Therefore, recalling $\Theta_0=sT$ and using the definition \eqref{eq:u-ucirc} of $u(x)$, we combine \eqref{eq:genus-zero-Y112} and \eqref{eq:genus-zero-product-Z} with \eqref{eq:genus-zero-outer-u} to obtain, for integers $m,n$ such that $(\Theta_0,\Theta_\infty)=(\Theta_{0,\mathrm{gO}}^{[3]}(m,n),\Theta_{\infty,\mathrm{gO}}^{[3]}(m,n))$,
\eq
\begin{split}
u^{[3]}_\mathrm{gO}(x;m,n)=u(x)&=-2sT\frac{Y^0_{0,11}(x)Y^0_{0,12}(x)}{Y^\infty_{1,12}(x)}\\
&=T^{\frac{1}{2}}(\gamma + \bo(T^{-1}))\\
& = 
|\Theta_0|^\frac{1}{2}\left(U_{0,\mathrm{gO}}(y;\kappa)+\bo(|\Theta_0|^{-1})\right),\quad
y=\frac{x}{|\Theta_0|^{\frac{1}{2}}},\quad\kappa=-\frac{\Theta_\infty}{|\Theta_0|}.
\end{split}
\label{eq:genus-zero-Okamoto-u-asymp}
\endeq
The fact that the same formula results in all cases comes from the correlation between $s=\pm 1$ and whether or not $C\subset\Sigma_h$ in each configuration.  

Likewise, using also the fact that the matrix elements of $\dot{\mathbf{O}}^\mathrm{out}(0)$ are bounded away from zero, 
\eq
\frac{Y^0_{0,11}(x)}{Y^0_{0,21}(x)} = \begin{cases}
\displaystyle \left(\tfrac{1}{4}T\right)^{\kappa T}\ee^{2Tg_\infty}
\left(\frac{\dot{O}^\mathrm{out}_{11}(0)}{\dot{O}^\mathrm{out}_{21}(0)}+\bo(T^{-1})\right),&\quad
C\not\subset\Sigma_h\\
\displaystyle\left(\tfrac{1}{4}T\right)^{\kappa T}\ee^{2Tg_\infty}
\left(\frac{\dot{O}^\mathrm{out}_{12}(0)}{\dot{O}^\mathrm{out}_{22}(0)}+\bo(T^{-1})\right),&\quad
C\subset\Sigma_h.
\end{cases}
\label{eq:genus-zero-quotient-Z}
\endeq
Therefore, recalling the definition of $u_\tw(x)$ in \eqref{eq:u-ucirc} and combining \eqref{eq:genus-zero-Y112} and \eqref{eq:genus-zero-quotient-Z} with either \eqref{eq:genus-zero-outer-ucirc-1} or \eqref{eq:genus-zero-outer-ucirc-2} we obtain, now for integers $m,n$ such that $(\Theta_{0,\mathrm{gO}}^{[1]}(m,n),\Theta_{\infty,\mathrm{gO}}^{[1]}(m,n))=(\Theta_{0,\tw},\Theta_{\infty,\tw})$,
\eq
\begin{split}
u^{[1]}_\mathrm{gO}(x;m,n)=u_\tw(x)&=-2\frac{Y^0_{0,21}(x)Y^\infty_{1,12}(x)}{Y^0_{0,11}(x)}\\
&=T^{\frac{1}{2}}\left(-y-\tfrac{1}{2}\gamma - 2s\gamma^{-1} + \bo(T^{-1})\right)\\
&= |\Theta_0|^\frac{1}{2}\left(-y-\tfrac{1}{2}U_{0,\mathrm{gO}}(y;\kappa)-2sU_{0,\mathrm{gO}}(y;\kappa)^{-1} + \bo(|\Theta_0|^{-1})\right),
\end{split}
\label{eq:genus-zero-Okamoto-ucirc-asymp}
\endeq
where $y$ and $\kappa$ are exactly as in \eqref{eq:genus-zero-Okamoto-u-asymp}.
Here, the only evidence of the four-fold origin of this asymptotic formula is the sign $s=\pm 1$.

Now, in both asymptotic formul\ae\ \eqref{eq:genus-zero-Okamoto-u-asymp} and \eqref{eq:genus-zero-Okamoto-ucirc-asymp}, $\gamma=U_{0,\mathrm{gO}}(y;\kappa)$ is the branch of the solution of the quartic \eqref{eq:gamma-eqn} satisfying $U_{0,\mathrm{gO}}(y;\kappa)=-\tfrac{2}{3}y + \bo(1)$ as $y\to\infty$.  However, recalling Remark~\ref{rem:other-parameters}, the parameters $(\Theta_0,\Theta_\infty)\in \Lambda_\mathrm{gO}$ appearing in these formul\ae\ are those for which $u(x)$ satisfies the Painlev\'e-IV equation, while $u_\tw(x)$ solves the Painlev\'e-IV equation for different parameters $(\Theta_{0,\tw},\Theta_{\infty,\tw})$; see \eqref{eq:Baecklund-3-to-1}.  Since the latter definition implies that $|\Theta_{0,\tw}|=\tfrac{1}{2}|\Theta_0|(1-s\kappa)$ and $1-s\kappa\neq 0$ for $\kappa\in (-1,1)$, it makes sense to write \eqref{eq:genus-zero-Okamoto-ucirc-asymp} in the form
\eq
u^{[1]}_\mathrm{gO}(x;m,n)=u_\tw(x)=|\Theta_{0,\tw}|^\frac{1}{2}C_\tw\left(-\frac{y}{C_\tw}-\tfrac{1}{2}U_{0,\mathrm{gO}}\left(\frac{y}{C_\tw};\kappa\right)-2sU_{0,\mathrm{gO}}\left(\frac{y}{C_\tw};\kappa\right)^{-1} + \bo(|\Theta_{0,\tw}|^{-1})\right),
\endeq
i.e., replacing $y$ with $y/C_\tw$ and using $|\Theta_0|^\frac{1}{2}=|\Theta_{0,\tw}|^\frac{1}{2}C_\tw$, where now $y$ is related to $x$ differently:
\eq
y:=\frac{x}{|\Theta_{0,\tw}|^\frac{1}{2}}\quad \text{and}\quad C_\tw:=\sqrt{\frac{2}{1-s\kappa}}.
\endeq
Next, we observe the following.
\begin{lemma}
Fix $\kappa\in (-1,1)$ and suppose that $\gamma=\gamma(y)$ is a solution of the quartic equation \eqref{eq:gamma-eqn} analytic on a domain $D$, i.e., $y\mapsto \gamma(y)$ is analytic on $D$ and $Q(\gamma(y),y;\kappa)=0$ holds identically on $D$.  Set
\eq
\gamma_\tw(y):=
C_\tw\left(-\frac{y}{C_\tw}-\tfrac{1}{2}\gamma\left(\frac{y}{C_\tw}\right)-2s\gamma\left(\frac{y}{C_\tw}\right)^{-1}\right).
\label{eq:gamma-circ}
\endeq
Then for $s=\pm 1$,
\eq
Q\left(\gamma_\tw(y),y;-\frac{\kappa+3s}{1-s\kappa}\right)=0
\endeq
holds for $y\in C_\tw D$ (dilation of the domain $D$ by $C_\tw$).
\label{lem:gamma-identity}
\end{lemma}
\begin{proof}
This follows from the identity
\eq
Q\left(C_\tw (-y-\tfrac{1}{2}\gamma-2s\gamma^{-1}),C_\tw y;-\frac{\kappa+3s}{1-s\kappa}\right)=\frac{\gamma^4-4(y^2+2\kappa)\gamma^2-32sy\gamma-48}{4\gamma^4(1-s\kappa)^2}Q(\gamma,y;\kappa).
\endeq
\end{proof}
One can easily check that if $\gamma(y)=-\tfrac{2}{3}y+\bo(1)$ in \eqref{eq:gamma-circ}, then
also $\gamma_\tw(y)=-\frac{2}{3}y + \bo(1)$.  Therefore, using Lemma~\ref{lem:gamma-identity} we can finally express \eqref{eq:genus-zero-Okamoto-ucirc-asymp} in the form
\eq
u^{[1]}_\mathrm{gO}(x;m,n)=u_\tw(x)=|\Theta_{0,\tw}|^\frac{1}{2}\left(U_{0,\mathrm{gO}}(y;\kappa_\tw) + \bo(|\Theta_{0,\tw}|^{-1})\right),\quad
y:=\frac{x}{|\Theta_{0,\tw}|^\frac{1}{2}},
\quad
\kappa_\tw =  -\frac{\Theta_{\infty,\tw}}{|\Theta_{0,\tw}|},
\label{eq:genus-zero-Okamoto-ucirc-asymp-2}
\endeq
because $U_{0,\mathrm{gO}}(y;\kappa)$ is a simple root of the quartic \eqref{eq:gamma-eqn}, and 
\eq
-\frac{\Theta_{\infty,\tw}}{|\Theta_{0,\tw}|} = -\frac{3\Theta_0-\Theta_\infty+2}{|\Theta_0+\Theta_\infty|} = -\frac{\kappa+3s}{1-s\kappa}+\bo(|\Theta_{0,\tw}|^{-1}).
\endeq

According to \eqref{eq:symmetry-1-2}, to prove the asymptotic formula \eqref{eq:OkamotoExterior123} in Theorem~\ref{thm:OkamotoExterior}, it suffices to prove it just for types $j=1$ and $j=3$.  
%To prove \eqref{eq:OkamotoExterior123} for $j=3$, recall the substitutions \eqref{eq:gO-type3-RHP-parameters}, which imply that $T:=|\Theta_0|=\tfrac{1}{2}|m+n|(1+\bo(|m+n|^{-1}))$ and $\kappa=-\Theta_\infty/|\Theta_0|=\mathrm{sgn}(m+n)(1-\rho)/(1+\rho)+\bo(T^{-1})$.  Hence using \eqref{eq:genus-zero-Okamoto-u-asymp} and noting again that the equilibrium $U_{0,\mathrm{gO}}(y;\kappa)$ is a simple root of \eqref{eq:equilibrium}, the proof of \eqref{eq:OkamotoExterior123} is complete for $j=3$.  To prove \eqref{eq:OkamotoExterior123} for $j=1$, we proceed similarly, recalling instead the substitutions \eqref{eq:gO-type1-RHP-parameters}.  To 
Therefore, to
complete the proof of Theorem~\ref{thm:OkamotoExterior} it remains only to specify the precise values of $y$ where \eqref{eq:genus-zero-Okamoto-u-asymp} and \eqref{eq:genus-zero-Okamoto-ucirc-asymp-2} are valid respectively and to discuss the uniformity of convergence on the corresponding domains.  This will be done in Section~\ref{sec:OutsideDomain} below.

%
%Therefore, using $x=T^{1/2}y$ and the fact that $E_{j,kl}=O(T^{-1})$ as $T\to +\infty$, we obtain
%\eq
%u^{(m,n)}(x):=2Y^{(m,n)}_{1,22}(x)-2x-2\frac{Y^{(m,n)}_{2,12}(x)}{Y^{(m,n)}_{1,12}(x)}=T^{1/2}\left(-2y-\tfrac{1}{2}(\alpha+\beta)+O(T^{-1})\right),\quad T\to +\infty.
%\endeq
%Finally, recalling the first identity in \eqref{eq:case-iv-coefficient-match}, this reads simply
%\eq
%u^{(m,n)}(x)=T^{1/2}\left(\gamma + O(T^{-1})\right),\quad T\to +\infty,
%\label{eq:exterior-asymptotic}
%\endeq
%and we recall that $\gamma=\gamma^{(\kappa)}(y)$ is the analytic continuation from $y=\infty$ of the solution of the quartic \eqref{eq:gamma-eqn} for which $\gamma=-\tfrac{2}{3}y+O(1)$ as $y\to\infty$.  This matches exactly the expected ``equilibrium'' asymptotic described in Section~\ref{sec:scaling}.

\section{Asymptotic analysis of $\mathbf{M}(z)$ for sufficiently large $|y|$:  gH case}
\label{sec:Exterior-gH}
Now we develop a simplified version of the analysis from Section~\ref{sec:Exterior} applicable to the gH family.  We take $(\Theta_0,\Theta_\infty)\in\Lambda_\mathrm{gH}^{[3]+}$ so that $s=\mathrm{sgn}(\Theta_0)=1$ and $\kappa=-\Theta_\infty/|\Theta_0|\in (-1,1)$, and we assume that for $|y|$ sufficiently large with $0\le\arg(y)\le\tfrac{1}{2}\pi$, the polynomial $P(z)$ is again in case $\{211\}$.  However, now we select the solution of the quartic \eqref{eq:gamma-eqn} that satisfies $\gamma=U_{0,\mathrm{gH}}^{[3]}(y;\kappa)=-2y + \bo(1)$ as $y\to\infty$ (see Section~\ref{sec:equilibrium}).  Solving for $\alpha$ and $\beta$ from \eqref{eq:alpha-beta-E} gives (breaking permutation symmetry) $\alpha = (-2\kappa+2\ii\sqrt{1-\kappa^2})y^{-1}+\bo(y^{-2})$ and $\beta=(-2\kappa-2\ii\sqrt{1-\kappa^2})y^{-1}+\bo(y^{-2})$ as $y\to\infty$.  It is convenient here to parametrize $\kappa\in (-1,1)$ by $\kappa=\sin(\tfrac{1}{2}\varphi)$, $\varphi\in (-\pi,\pi)$.  Then also $\sqrt{1-\kappa^2}=\cos(\tfrac{1}{2}\varphi)>0$, so $\alpha=2\ii\ee^{\frac{1}{2}\ii\varphi}y^{-1}+\bo(y^{-2})$ and $\beta=-2\ii\ee^{-\frac{1}{2}\ii\varphi}y^{-1}+\bo(y^{-2})$ as $y\to\infty$.

\subsection{Analysis of the exponent $h(z)$}
For large $|y|$, the quadratic differential $Q(z)\,\dd z^2$ for $Q(z)=\tfrac{1}{16}z^{-2}(z-\gamma)^2(z-\alpha)(z-\beta)$ can be written under the rescaling $Z=yz$ as $Q(z)\,\dd z^2 = \tfrac{1}{4}Z^{-2}((Z+2\kappa)^2+4(1-\kappa^2))(1+\bo(y^{-1}))\,\dd Z^2$, 
where the error term is uniform for bounded $Z$.  Neglecting the error term yields a quadratic differential in the $Z$-plane that has Schwarz symmetry and only two finite critical points; hence both finite critical points necessarily lie on the boundary of the circle domain $\circledomain$ containing $Z=0$.  This leading-order model resolves the limiting v-trajectories in the part of the $z$-plane that asymptotically contains $z=0$, $z=\alpha$, and $z=\beta$, while $z=\gamma$ is out of the picture.  Restoring the error term, one can show that for large $|y|$ this structure is preserved and hence both $\alpha$ and $\beta$ lie on $\partial\circledomain$ while $\gamma\in\mathbb{C}\setminus\overline{\circledomain}$.  Therefore, $\partial\circledomain$ is the closure of the union of two v-trajectories, each with endpoints $z=\alpha,\beta$.  From each of the latter critical points, exactly one additional v-trajectory emanates into the exterior of $\partial\circledomain$, and since there can be no divergent v-trajectories by the same argument as in Section~\ref{sec:h-analysis}, these two v-trajectories can either coincide, terminate at $z=\gamma$, or tend to $z=\infty$.  The scenario of coincidence would imply a closed loop formed of v-trajectories that can be easily ruled out by Lemma~\ref{lem:Teichmueller}.  Without loss of generality, we assume that $\mathrm{Re}(h(\alpha))=0$.   

Suppose that $\mathrm{Re}(h(\gamma))\neq 0$. It then follows that the v-trajectories emanating from $z=\alpha,\beta$ into the exterior of the circle domain both tend to $z=\infty$, and the exterior of $\partial\circledomain$ is divided by these trajectories into two disjoint components, exactly one of which must contain $z=\gamma$ and the four critical v-trajectories emanating from it.  Therefore, the complement of the closure $\critclosure$ of the union of critical v-trajectories of $Q(z)\,\dd z^2$ is the disjoint union of four end domains, one strip domain, and one circle domain, as shown in 
Figure~\ref{fig:GenusZeroGH-Stokes}.
%the left-hand panel of Figure~\ref{fig:GenusZeroGH-first-two}.
%\begin{figure}[h]
%\begin{center}
%\includegraphics{GenusZeroGH-first-two.pdf}
%\end{center}
%\caption{For $\kappa=0.5$ and $y=1.3\ee^{\frac{\ii\pi}{4}}$ ($\gamma=\gamma_\mathrm{gH}^{[3]}(y;\kappa)$ analytically continued from large $y$), a situation for which $\mathrm{Re}(h(\gamma))\neq 0$.  Left:  the critical v-trajectories divide the plane into four end domains, one strip domain, and one circle domain.  Right:  The zero level set of $\mathrm{Re}(h(z))$ (black and orange curves), the branch cut for $h'(z)$ (orange), and sign of $\mathrm{Re}(h(z))$ (shaded for negative, unshaded for positive).}
%\label{fig:GenusZeroGH-first-two}
%\end{figure} 
\begin{figure}[h]
\begin{center}
\includegraphics{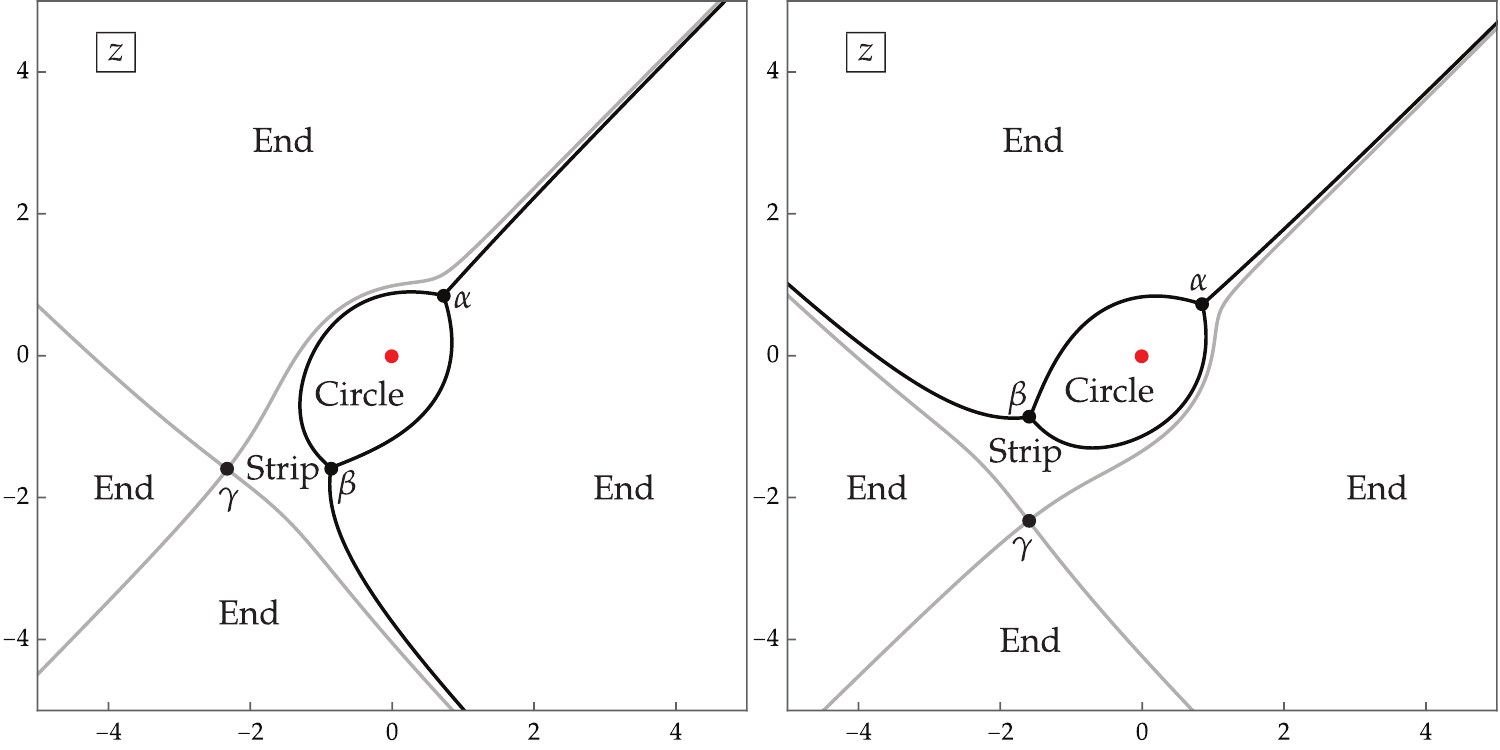}
\end{center}
\caption{For $\kappa=0$ and $y=1.198+0.983\ii$ (left) and $y=0.983+1.198\ii$ (right) ($\gamma=U_{0,\mathrm{gH}}^{[3]}(y;\kappa)$ analytically continued from large $y$), situations in which $\mathrm{Re}(h(\gamma))$ has opposite signs.  In both cases the critical v-trajectories divide the plane into four end domains, one strip domain, and one circle domain.}
\label{fig:GenusZeroGH-Stokes}
\end{figure} 

One important distinction from the gO case discussed in Section~\ref{sec:h-analysis} is that $\alpha$ and $\beta$ lie in the same connected component of $\critclosure$, as can be seen in 
%the left-hand panel of Figure~\ref{fig:GenusZeroGH-first-two}. 
Figure~\ref{fig:GenusZeroGH-Stokes}. 
Moreover $\alpha$ and $\beta$ are joined by two v-trajectories, which implies that in this case the component of $\critclosure$ containing $\alpha$ and $\beta$ is a strict subset of the level set $\mathrm{Re}(h(z))=0$, because the component contains only two unbounded arcs while the level set has four arcs that go to $z=\infty$ parallel to the four directions $\arg(z)=\pm\tfrac{1}{4}\pi,\pm\tfrac{3}{4}\pi$.  Since the two missing unbounded arcs of the level set tend to $z=\infty$ in different directions distinct from the direction of the unbounded arcs emanating from $\alpha$ and $\beta$ and cannot cross the arcs emanating from $z=\gamma$ because $\mathrm{Re}(h(\gamma))\neq 0$ by assumption, it follows that they are trapped within the end domain opposite $\gamma$ from $z=0\subset\circledomain$.  Since the end domain does not contain any critical points, the missing unbounded arcs of the level set actually form the same v-trajectory.  It can be seen near the left/bottom of the left-hand/right-hand panel of 
%Figure~\ref{fig:GenusZeroGH-first-two}, 
Figure~\ref{fig:GenusZeroGH-LevelAndCut},
\begin{figure}[h]
\begin{center}
\includegraphics{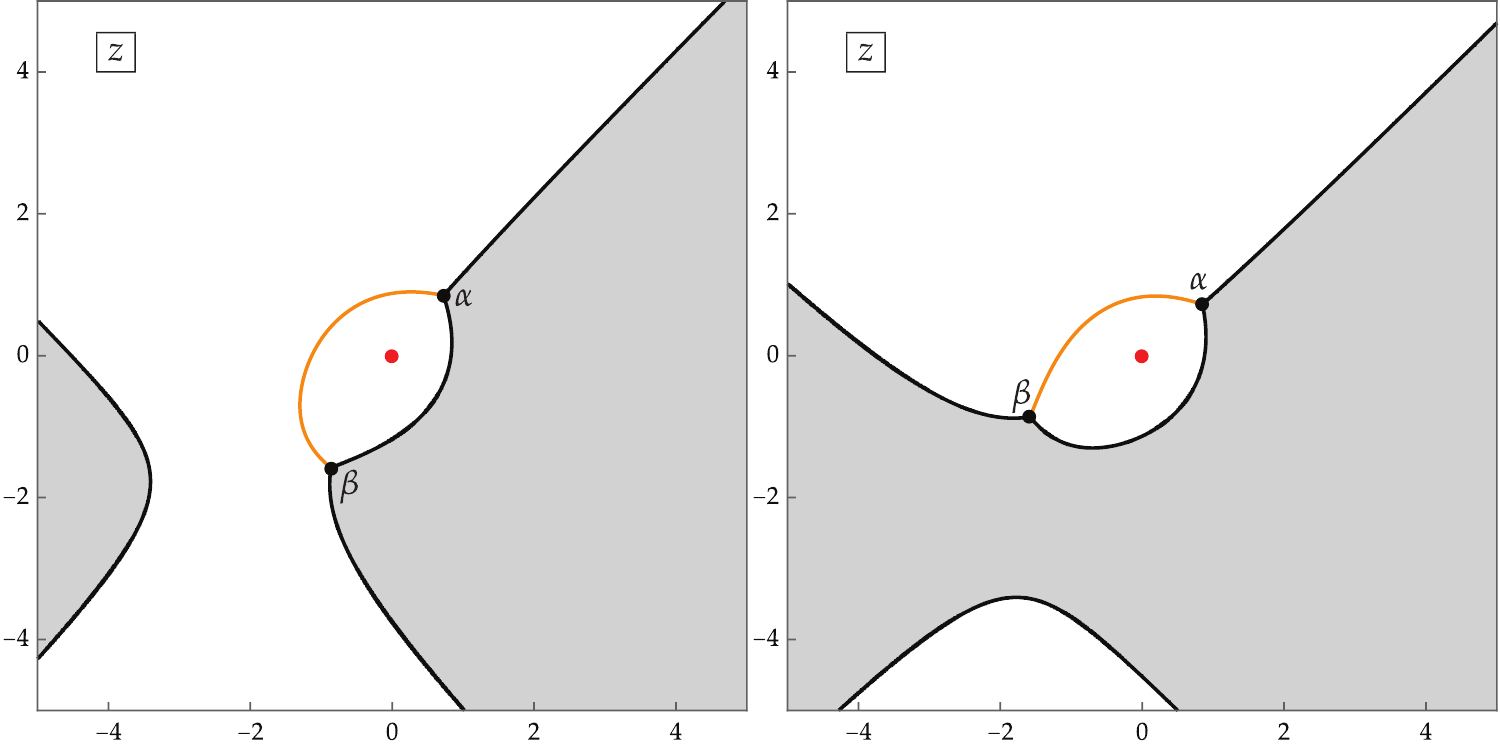}
\end{center}
\caption{For the same parameters as in the corresponding panels of Figure~\ref{fig:GenusZeroGH-Stokes}, the zero level set of $\mathrm{Re}(h(z))$ (black and orange curves), the branch cut for $h'(z)$ (orange), and sign of $\mathrm{Re}(h(z))$ (shaded for negative, unshaded for positive).}
\label{fig:GenusZeroGH-LevelAndCut}
\end{figure} 
and its disjoint union with the component of $\critclosure$ containing $\alpha$ and $\beta$ is precisely the level set $\mathrm{Re}(h(z))=0$.

Now we consider the possibility that $\mathrm{Re}(h(\gamma))=0$.  A calculation parallel to \eqref{eq:genus-zero-h-of-gamma-Okamoto} shows that
\eq
\mathrm{Re}(h(\gamma))=\mathrm{Re}\left(\int_{\alpha}^\gamma h'(z)\,\dd z\right)=\pm\frac{1}{2}\left(\mathrm{Re}(y)^2-\mathrm{Im}(y)^2\right)+o(y^2),\quad y\to\infty,
\label{eq:genus-zero-h-of-gamma-Hermite}
\endeq
so that $\mathrm{Re}(h(\gamma))$ cannot vanish for large $|y|$ unless $\arg(y)\approx \tfrac{1}{4}\pi\in [0,\tfrac{1}{2}\pi]$.  If $\mathrm{Re}(h(\gamma))=0$, then the topology of $\critclosure$ is different.  In particular $\critclosure$ becomes connected and it now coincides with the level set $\mathrm{Re}(h(z))=0$.  However we observe that if $|y|$ is large, then since $\gamma$ is large while $\circledomain$ is small, the condition $\mathrm{Re}(h(\gamma))=0$ can only occur \emph{without $\gamma$ lying on $\partial\circledomain$}.  Therefore, the topology of the level set $\mathrm{Re}(h(z))=0$ in a neighborhood of the circle domain $\circledomain$ is just as if $\mathrm{Re}(h(\gamma))\neq 0$.  We touch on this observation again later in Section~\ref{sec:ExteriorAxes}.

Given the structure of the level set $\mathrm{Re}(h(z))=0$ near $\circledomain$ as shown in 
%the right-hand panel of Figure~\ref{fig:GenusZeroGH-first-two}, 
Figure~\ref{fig:GenusZeroGH-LevelAndCut},
we now explain how to determine $h'(z)$ and then $h(z)$ precisely.  Unlike in the gO case discussed in Section~\ref{sec:h-analysis}, we can and will take the branch cut $\Sigma_\mathrm{c}$ of $R(z)$ to coincide with one of the two v-trajectories connecting $\alpha$ and $\beta$; because $s=1$, we need to select the specific v-trajectory to use so that $R(0)=4s=4$.
In the special case that $\kappa=0$ and $\arg(y)=\tfrac{1}{4}\pi$, it is easy to see that for $|y|$ sufficiently large $\alpha$, $\beta$, and $\gamma$ all lie on the diagonal line through the origin, with $\arg(\alpha)=\tfrac{1}{4}\pi$ and $\arg(\beta)=\arg(\gamma)=-\tfrac{3}{4}\pi$.  This implies that $R(z)^2$ is real for $z$ along the same diagonal line, and since by definition $R(z)=z^2+\bo(z)$ as $z\to\infty$, $R(z)$ is positive imaginary for $\arg(z)=\tfrac{1}{4}\pi$ and $|z|>|\alpha|$.  It then follows that to have $R(0)>0$ we must choose $\Sigma_\mathrm{c}$ to lie in the half-plane above the diagonal line: $\mathrm{Im}(z)\ge\mathrm{Re}(z)$.  This is the unique v-trajectory on the Jordan curve $\partial\circledomain$ that abuts a region exterior to $\partial\circledomain$ on which $\mathrm{Re}(h(z))>0$, and this topological characterization of $\Sigma_\mathrm{c}$ is robust as $\kappa$ and $y$ vary.  The branch cut $\Sigma_\mathrm{c}$ for $R(z)$ is shown with an orange curve in 
%the right-hand panel of Figure~\ref{fig:GenusZeroGH-first-two}.
each panel of Figure~\ref{fig:GenusZeroGH-LevelAndCut}.
Once $R(z)$ is determined, then so is $h'(z)$ by \eqref{eq:h-g-phi}.  Accounting for the pole of $h'(z)$ at $z=0$, we choose the point $z=\beta$ to be the common endpoint of the arcs $\Sigma_0$ and $\Sigma_{4,3}$ and then we take the jump contour for $h(z)$ to be $\Sigma_h:=\Sigma_\mathrm{c}\cup\Sigma_0\cup\Sigma_{4,3}$.  Finally, we define $\mathbb{C}\setminus\Sigma_h\ni z\mapsto h(z)$ by integration of $h'(\cdot)$ from $\alpha$ to $z$ over any path lying in $\mathbb{C}\setminus\Sigma_h$.
Note that while in the gO case $\mathrm{Re}(h(z))$ exhibits a jump discontinuity across $\Sigma_\mathrm{c}$, in this case $\mathrm{Re}(h(z))$ extends to $\Sigma_\mathrm{c}$ as a continuous function as a consequence of choosing $\Sigma_\mathrm{c}$ as a zero level curve of $\mathrm{Re}(h(z))$.  The sign of $\mathrm{Re}(h(z))$ is as indicated with shading in 
%the right-hand panel of Figure~\ref{fig:GenusZeroGH-first-two}.
Figure~\ref{fig:GenusZeroGH-LevelAndCut}.

The analytic function $h(z)$ defined in this way takes continuous boundary values on $\Sigma_h$ that are related by (with the orientation of the arcs $\Sigma_0$ and $\Sigma_{4,3}$ as indicated in Figure~\ref{fig:PIV-Sigma}) 
\eq
\begin{split}
\Delta h(z)&=-2\pi\ii,\quad z\in\Sigma_0,\\
\langle h\rangle (z)&= 0,\quad z\in\Sigma_\mathrm{c},\\
\Delta h(z)&=2\pi\ii\kappa,\quad z\in\Sigma_{4,3}.
\end{split}
\label{eq:h-jumps-GenusZeroGH}
\endeq

\subsection{Use of $g(z)$ to convert $\mathbf{M}(z)$ to $\mathbf{O}(z)$}
We proceed to lay the (suitably deformed and fixed in the $z$-plane) jump contour $\Sigma_\mathrm{gH}$ from Figure~\ref{fig:PIV-Sigma-Hermite} over the sign chart of $\mathrm{Re}(h(z))$ as shown in 
%the left-hand panel of Figure~\ref{fig:GenusZeroGH-second-two}.
Figure~\ref{fig:GenusZeroGH-MContour}.
%\begin{figure}[h]
%\begin{center}
%\includegraphics{GenusZeroGH-second-two.pdf}
%\end{center}
%\caption{For $\kappa=0.5$ and $y=1.3\ee^{\frac{\ii\pi}{4}}$.  Left:  the jump contour for $\mathbf{M}(z)$. Right:  The ``lens'' domains $\Lambda^\pm$ and the modified jump contour for $\mathbf{O}(z)$ (the dashed arcs have been removed by the transformation $\mathbf{M}(z)\mapsto \mathbf{N}(z)$).}
%\label{fig:GenusZeroGH-second-two}
%\end{figure}
\begin{figure}[h]
\begin{center}
\includegraphics{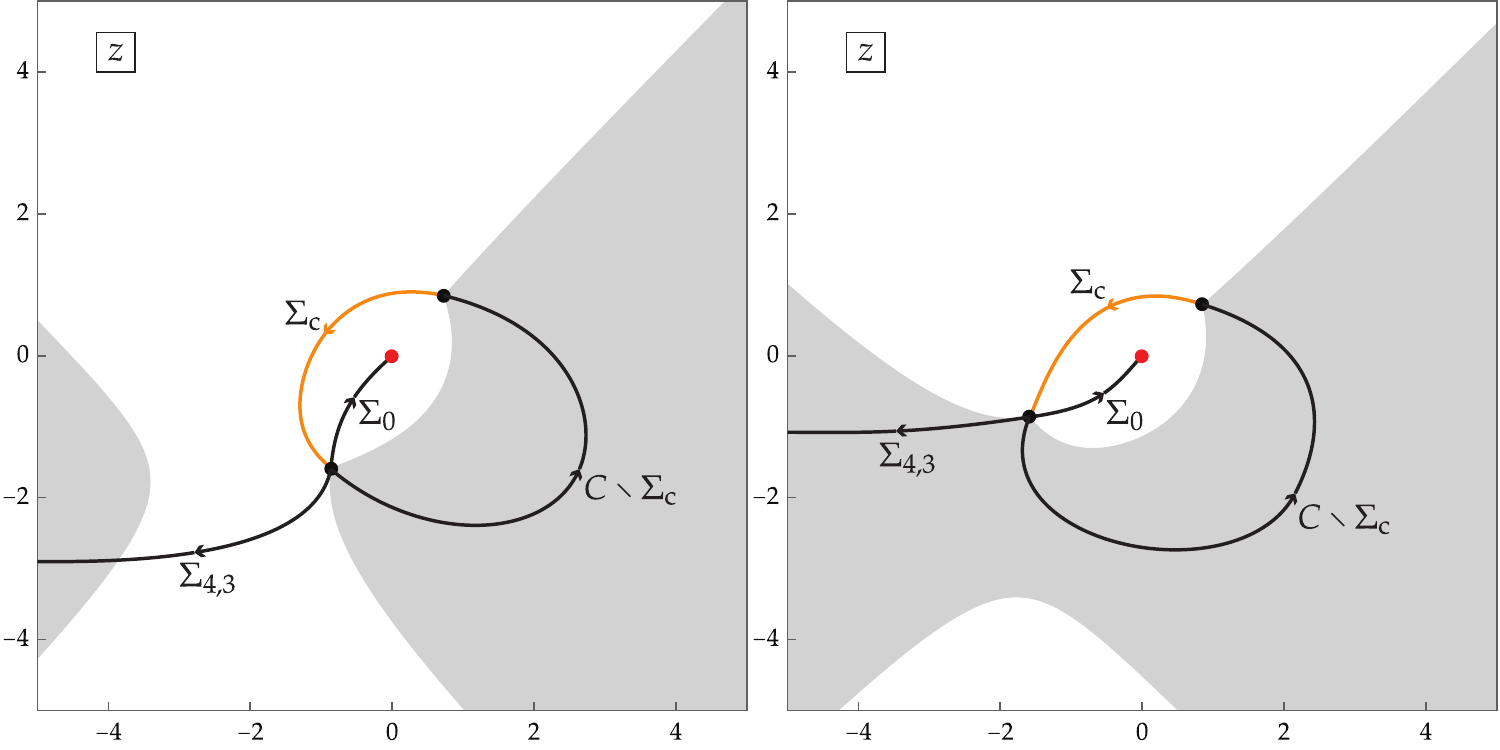}
\end{center}
\caption{For the same parameters as in the corresponding panels of Figures~\ref{fig:GenusZeroGH-Stokes} and \ref{fig:GenusZeroGH-LevelAndCut}, the jump contour for $\mathbf{M}(z)$.}
\label{fig:GenusZeroGH-MContour}
\end{figure}
In particular, we take $\Sigma_\mathrm{c}\subset C\subset\Sigma_\mathrm{gH}$ and insist that $\mathrm{Re}(h(z))<0$ holds on $C\setminus\Sigma_\mathrm{c}$.  Then we define $g(z)$ from $h(z)$ by \eqref{eq:h-g-phi} and use it to define $\mathbf{N}(z)$ from $\mathbf{M}(z)$ by \eqref{eq:NfromM}.  Using \eqref{eq:h-jumps-GenusZeroGH} it is easy to check that this transformation removes the jump discontinuities from the arcs $\Sigma_0\cup\Sigma_{4,3}$, so that $\mathbf{N}(z)$ is analytic for $z\in\mathbb{C}\setminus C$ and $\mathbf{N}(z)\to\mathbb{I}$ as $z\to\infty$.  Its jump conditions are
\eq
\mathbf{N}_+(z)=\mathbf{N}_-(z)\bpm \ee^{T\Delta h(z)} & 0\\\ee^{2T\langle h\rangle(z)} & \ee^{-T\Delta h(z)}\epm,\quad z\in\Sigma_\mathrm{c}
\endeq
and
\eq
\mathbf{N}_+(z)=\mathbf{N}_-(z)
%\bpm 1 & 0\\\ee^{2Th(z)} & 1\epm,
\mathbf{L}(\ee^{2Th(z)}),
\quad z\in C\setminus\Sigma_\mathrm{c}.
\endeq
Applying a ``UTU'' factorization (see \eqref{eq:general-factorizations}) the jump matrix on $\Sigma_\mathrm{c}$ can be written in the form
\eq
\begin{split}
\bpm \ee^{T\Delta h(z)} & 0\\\ee^{2T\langle h\rangle(z)} & \ee^{-T\Delta h(z)}\epm = \ee^{-Th_-(z)\sigma_3}\bpm 1&0\\1&1\epm\ee^{Th_+(z)\sigma_3} &= \ee^{-Th_-(z)\sigma_3}\mathbf{U}(1)\mathbf{T}(1)\mathbf{U}(1)\ee^{Th_+(z)\sigma_3} \\ &= \mathbf{U}(\ee^{-2Th_-(z)})\mathbf{T}(\ee^{2T\langle h\rangle(z)})\mathbf{U}(\ee^{-2Th_+(z)}) \\ &= \mathbf{U}(\ee^{-2Th_-(z)})\mathbf{T}(1)\mathbf{U}(\ee^{-2Th_+(z)})
\end{split}
\endeq
where in the last step we used \eqref{eq:h-jumps-GenusZeroGH}.  Based on this factorization, we introduce lens domains $\Lambda^+$ and $\Lambda^-$ on the left and right, respectively, of $\Sigma_\mathrm{c}$ as shown in 
%the right-hand panel of Figure~\ref{fig:GenusZeroGH-second-two}.  
Figure~\ref{fig:GenusZeroGH-Lenses}.
\begin{figure}[h]
\begin{center}
\includegraphics{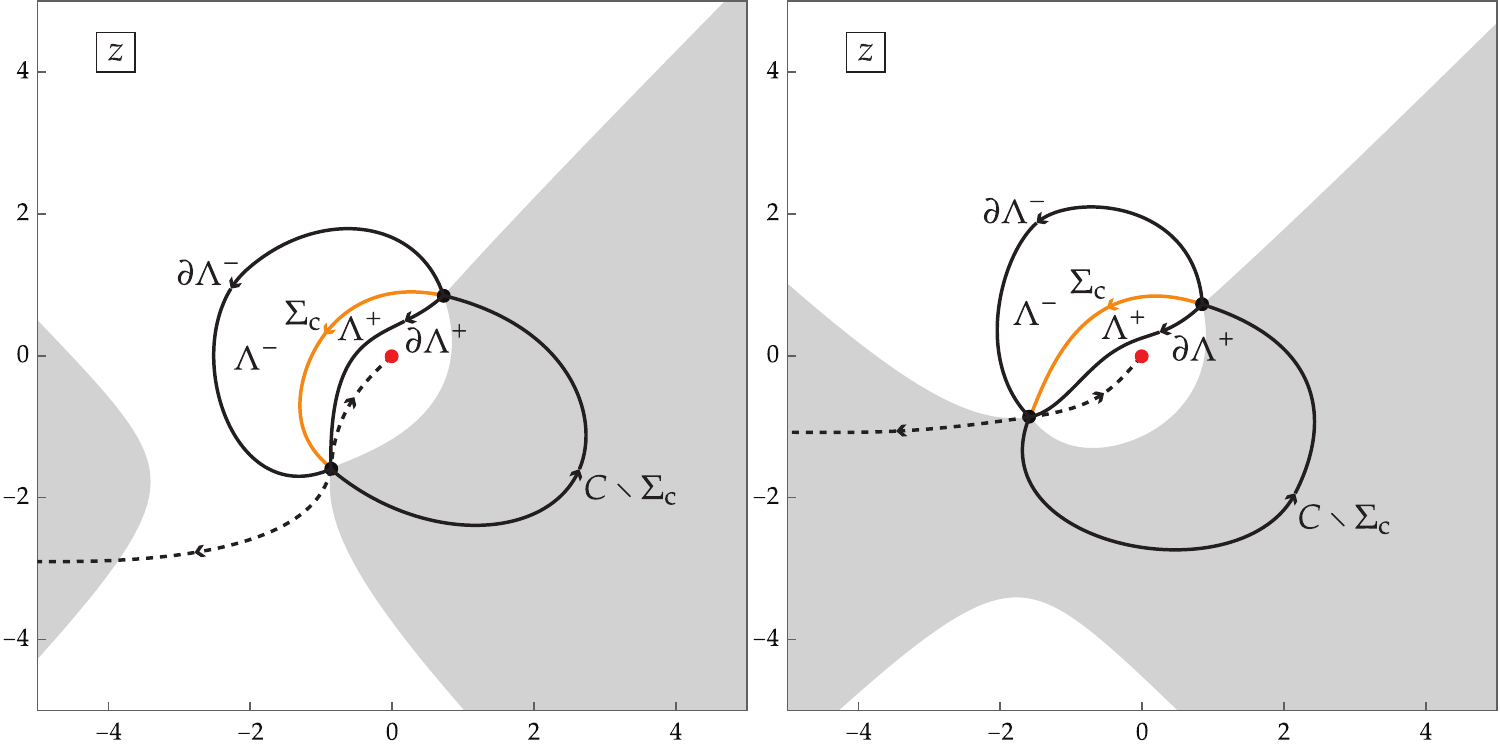}
\end{center}
\caption{For the same parameters as in the corresponding panels of Figures~\ref{fig:GenusZeroGH-Stokes}-- \ref{fig:GenusZeroGH-MContour}, the ``lens'' domains $\Lambda^\pm$ and the modified jump contour for $\mathbf{O}(z)$ (the dashed arcs have been removed by the transformation $\mathbf{M}(z)\mapsto \mathbf{N}(z)$).}
\label{fig:GenusZeroGH-Lenses}
\end{figure}
Then we define a new unknown matrix $\mathbf{O}(z)$ in terms of $\mathbf{N}(z)$ by setting
\eq
\mathbf{O}(z):=\begin{cases}
\mathbf{N}(z)\mathbf{U}(\ee^{-2Th(z)})^{-1},\quad & z\in\Lambda^+\\
\mathbf{N}(z)\mathbf{U}(\ee^{-2Th(z)}),\quad & z\in\Lambda^-\\
\mathbf{N}(z),\quad&\text{elsewhere},
\end{cases}
\endeq
and it follows that the three factors in the jump of $\mathbf{N}(z)$ across $\Sigma_\mathrm{c}$ are split into separate jumps of $\mathbf{O}(z)$ across three arcs with the same endpoints:
\eq
\mathbf{O}_+(z)=\mathbf{O}_-(z)\mathbf{U}(\ee^{-2Th(z)}),\quad z\in\partial\Lambda^\pm,
\endeq
and
\eq
\mathbf{O}_+(z)=\mathbf{O}_-(z)\mathbf{T}(1),\quad z\in\Sigma_\mathrm{c}.
\endeq 

\subsection{Parametrix construction}
\subsubsection{Outer parametrix}
Because $C\setminus\Sigma_\mathrm{c}$ lies in a region where $\mathrm{Re}(h(z))<0$ while the lens boundaries $\partial\Lambda^\pm$ both lie in regions where $\mathrm{Re}(h(z))>0$, it appears reasonable to neglect the jumps on these arcs.  We therefore define an outer parametrix $\dot{\mathbf{O}}^\mathrm{out}(z)$ to be analytic except on $\Sigma_\mathrm{c}$ across which it satisfies the same jump condition as does $\mathbf{O}(z)$, and we insist that $\dot{\mathbf{O}}^\mathrm{out}(z)$ be bounded except near $z=\alpha,\beta$ where negative one-fourth power divergences are admitted and that $\dot{\mathbf{O}}^\mathrm{out}(z)\to\mathbb{I}$ as $z\to\infty$.  Thus, the outer parametrix is explicitly given by (cf.\@ \eqref{eq:OdotOut-genus-zero})
\eq
\dot{\mathbf{O}}^\mathrm{out}(z)=\mathbf{S}_\mathrm{gH}j(z)^{\sigma_3}\mathbf{S}_\mathrm{gH}^{-1},\quad
\mathbf{S}_\mathrm{gH}:=\bpm \tfrac{1}{2} & \ii\\\tfrac{1}{2}\ii & 1\epm,
\label{eq:GenusZero-Outer-GH}
\endeq
where the function $j(z)$ is defined exactly as in Section~\ref{sec:Genus-Zero-Parametrix}.  This formula can be obtained from \eqref{eq:OdotOut-genus-zero} using a conjugation by a constant diagonal matrix, and therefore it follows immediately that all four equations in \eqref{eq:genus-zero-outer-infty}--\eqref{eq:genus-zero-outer-zero} are modified only by a common factor.  The identity \eqref{eq:GenusZero-squares-identity} holds exactly as written, but it requires a reinterpretation because $\alpha$, $\beta$, and $\gamma$ are different functions of $y$ in the gH and gO cases.  In the gH case, we have $j(0)^4=\alpha/\beta=-\ee^{\ii\varphi}+\bo(y^{-1})$ as $y\to\infty$.  To take the square root correctly it is easiest to take $\kappa=0$ and $\arg(y)=\tfrac{1}{4}\pi$ to arrange $\alpha$ and $\beta$ along the diagonal with the branch cut for $j(z)$ lying above the diagonal.  Then it is easy to see that $j(0)^2=-\ii\ee^{\frac{1}{2}\ii\varphi}+\bo(y^{-1})$ as $y\to\infty$.  Therefore $j(0)^2\pm j(0)^{-2}=-\ii\ee^{\frac{1}{2}\ii\varphi}\pm\ii\ee^{-\frac{1}{2}\ii\varphi}+\bo(y^{-1})$ for large $y$.  Since also from the large-$y$ expansions of $\alpha$, $\beta$, and $\gamma$ in the gH case we have $\tfrac{1}{4}\gamma(\alpha\pm\beta)=-\ii\ee^{\frac{1}{2}\ii\varphi}\pm\ii\ee^{-\frac{1}{2}\ii\varphi}+\bo(y^{-1})$, the exact identity $j(0)^2\pm j(0)^{-2}=\tfrac{1}{4}\gamma(\alpha\pm\beta)$ holds in place of the more complicated formula \eqref{eq:f(0)squared-identity-gO}.  It follows that 
both formul\ae\ \eqref{eq:genus-zero-outer-u}--\eqref{eq:genus-zero-outer-ucirc-1} are valid in the gH case as well if we take the first line corresponding to the ``leftward'' configuration.  

\subsubsection{Inner parametrices}
Fixing disks $D_\alpha$ and $D_\beta$ containing $z=\alpha$ and $z=\beta$ respectively, we can define conformal maps $z\mapsto W(z)$ on each as shown in Table~\ref{tab:Genus-Zero-Hermite-Airy}.  Taking the constant matrix $\mathbf{C}$ as given in the same table, we again use the formula \eqref{eq:GenusZero-H-define} to define from the conformal map $W(z)$ and the outer parametrix (now given by \eqref{eq:GenusZero-Outer-GH}) a holomorphic unit determinant matrix $\mathbf{H}(z)$ on each disk.  Then we use \eqref{eq:GenusZero-inner-parametrices-define} to define inner parametrices $\dot{\mathbf{O}}^{\mathrm{in},\alpha}(z)$ and $\dot{\mathbf{O}}^{\mathrm{in},\beta}(z)$ on $D_\alpha$ and $D_\beta$ respectively.  
\begin{table}[h]
\caption{Data for defining the inner parametrices in $D_\alpha$ and $D_\beta$.}
\renewcommand{\arraystretch}{0.45}
\begin{tabular}{@{}|l|l|c|c|c|@{}}
\hline
\multirow{2}{*}{$p$}&\multirow{2}{*}{Conformal map $W:D_p\to\mathbb{C}$}&\multicolumn{2}{c|}{Ray Preimages in $D_p$} & \multirow{2}{*}{Constant matrix $\mathbf{C}$}\\
\cline{3-4}
&  & $\arg(W)$ & Preimage &  \\
\hline\hline
\multirow{4}{*}{$\alpha$} &\multirow{4}{*}{\makecell[l]{$(-2h(z))^{\frac{2}{3}}$,\\ 
continued from $C\setminus\Sigma_\mathrm{c}$}} &
\shortstrut $0$ & $C\setminus\Sigma_\mathrm{c}$ & 
\multirow{4}{*}{$\ee^{\frac{\ii\pi}{4}\sigma_3}$}  \\
\cline{3-4}
&& \shortstrut $\tfrac{2}{3}\pi$ & $\partial\Lambda^-$ & \\
\cline{3-4}
&& \shortstrut $-\tfrac{2}{3}\pi$ & $\partial\Lambda^+$ & \\
\cline{3-4}
&& \shortstrut $\pm\pi$ & $\Sigma_\mathrm{c}$ & \\
\hline
\hline
\multirow{4}{*}{$\beta$} &\multirow{4}{*}{\makecell[l]{$(-2h(z))^{\frac{2}{3}}$,\\ 
continued from $C\setminus\Sigma_\mathrm{c}$}} &
\shortstrut $0$ & $C\setminus\Sigma_\mathrm{c}$ & 
\multirow{4}{*}{$\ee^{-\frac{\ii\pi}{4}\sigma_3}$}  \\
\cline{3-4}
&& \shortstrut $\tfrac{2}{3}\pi$ & $\partial\Lambda^+$ & \\
\cline{3-4}
&& \shortstrut $-\tfrac{2}{3}\pi$ & $\partial\Lambda^-$ & \\
\cline{3-4}
&& \shortstrut $\pm\pi$ & $\Sigma_\mathrm{c}$ & \\
\hline
\end{tabular}
\renewcommand{\arraystretch}{1}
\label{tab:Genus-Zero-Hermite-Airy}
\end{table}

Again, these inner parametrices are exact local solutions within $D_\alpha$ and $D_\beta$ of the analyticity and jump conditions to be satisfied by $\mathbf{O}(z)$, and it follows from the construction that the estimates \eqref{eq:GenusZero-Mismatch} are valid in the gH case as well.

\subsubsection{Global parametrix and error estimation}
We adopt the same definition \eqref{eq:GenusZero-global} as in the gO case for the global parametrix $\dot{\mathbf{O}}(z)$ in terms of the (slightly modified) outer and inner parametrices just described.  The analysis of the error $\mathbf{E}(z)$ given by \eqref{eq:GenusZeroE} outlined in Section~\ref{sec:GenusZero-error} is soft and it applies equally well in the present context with the same results; the expansion \eqref{eq:GenusZero-E-expansion} holds with $\mathbf{E}_1=\bo(T^{-1})$ and $\mathbf{E}(0)=\mathbb{I}+\bo(T^{-1})$.

\subsection{Asymptotic formul\ae\ for the gH rational solutions of Painlev\'e-IV on the exterior domain}
The relationship between $\mathbf{Y}(\lambda;x)$ and the product $\mathbf{E}(z)\dot{\mathbf{O}}^\mathrm{out}(z)$ for large $\lambda$ or large $z$ is exactly as in the gO case, and hence the formula \eqref{eq:genus-zero-Y112} for $Y^\infty_{1,12}(x)$ holds in the gH case as well.  Likewise, the exact formula \eqref{eq:GenusZero-Z} for $\mathbf{Y}^0_0(x)$ is valid; however in the gH case there is substantial simplification in expressing $\mathbf{N}(0)$ in terms of known or estimable quantities.  Indeed, in place of \eqref{eq:GenusZero-N(0)} we have simply $\mathbf{N}(0)=\mathbf{O}(0)$, and then as in the gO case, 
$\mathbf{O}(0)=(\mathbb{I}+\bo(T^{-1}))\dot{\mathbf{O}}^\mathrm{out}(0)$.  Hence the formula for $Y^0_{0,11}(x)Y^0_{0,12}(x)$ matches the first line of \eqref{eq:genus-zero-product-Z}; combining this with \eqref{eq:genus-zero-Y112}, using the first line of \eqref{eq:genus-zero-outer-u} and taking into account $s=1$ in the definition \eqref{eq:u-ucirc} we find
\eq
u^{[3]}_\mathrm{gH}(x;m,n)=u(x)=T^{\frac{1}{2}}(\gamma+\bo(T^{-1}))=|\Theta_0|^\frac{1}{2}(U_{0,\mathrm{gH}}^{[3]}(y;\kappa)+\bo(|\Theta_0|^{-1})),\quad y:=\frac{x}{|\Theta_0|^\frac{1}{2}},\quad\kappa:=-\frac{\Theta_\infty}{|\Theta_0|}.
\label{eq:genus-zero-Hermite-u-asymp}
\endeq
Likewise, the formula for the fraction $Y^0_{0,11}(x)/Y^0_{0,21}(x)$ matches the first line of \eqref{eq:genus-zero-quotient-Z}.  Combining this with \eqref{eq:genus-zero-Y112} and using the first line of \eqref{eq:genus-zero-outer-ucirc-1} in the definition \eqref{eq:u-ucirc} gives
\eq
\begin{split}
u^{[1]}_\mathrm{gH}(x;m,n)=u_\tw(x)&=T^\frac{1}{2}\left(-y-\tfrac{1}{2}\gamma-2\gamma^{-1}+\bo(T^{-1})\right)\\ & = 
|\Theta_0|^\frac{1}{2}\left(-y-\tfrac{1}{2}U_{0,\mathrm{gH}}^{[3]}(y;\kappa)-2U_{0,\mathrm{gH}}^{[3]}(y;\kappa)^{-1}+\bo(|\Theta_0|^{-1})\right).
\end{split}
\endeq
We can apply Lemma~\ref{lem:gamma-identity} in the case $s=1$ to write this formula in a more convenient form.  Noting that if $\gamma(y)$ is the branch of the quartic \eqref{eq:gamma-eqn} that behaves like $\gamma(y)=-2y + 2\kappa y^{-1}+\bo(y^{-3})$ then the related function $\gamma_\tw(y)$ defined by \eqref{eq:gamma-circ} obeys $\gamma_\tw(y)=2y^{-1}+\bo(y^{-3})$ as $y\to\infty$,
we arrive at the following:
\eq
u^{[1]}_\mathrm{gH}(x;m,n)=u_\tw(x)=|\Theta_{0,\tw}|^\frac{1}{2}\left(U_{0,\mathrm{gH}}^{[1]}(y;\kappa_\tw)+\bo(|\Theta_{0,\tw}|^{-1})\right),\quad
y:=\frac{x}{|\Theta_{0,\tw}|^{\frac{1}{2}}},\quad\kappa_\tw:=-\frac{\Theta_{\infty,\tw}}{|\Theta_{0,\tw}|},
\label{eq:genus-zero-Hermite-ucirc-asymp}
\endeq
where $\gamma=U_{0,\mathrm{gH}}^{[1]}(y;\kappa)$ is the solution of the quartic \eqref{eq:gamma-eqn} with asymptotic behavior $U_{0,\mathrm{gH}}^{[1]}(y;\kappa)=2y^{-1}+\bo(y^{-3})$ as $y\to\infty$.  See Section~\ref{sec:equilibrium}.

The asymptotic formul\ae\ \eqref{eq:genus-zero-Hermite-u-asymp} and \eqref{eq:genus-zero-Hermite-ucirc-asymp} provide the proofs of \eqref{eq:HermiteExterior123} for $j=3$ and $j=1$ respectively following nearly identical reasoning as described at the end of Section~\ref{sec:ExteriorFormulasOkamoto} (and by the symmetry \eqref{eq:symmetry-1-2} this is enough to establish also \eqref{eq:HermiteExterior123} for $j=2$).  This establishes Theorem~\ref{thm:HermiteExterior}, aside from the statements concerning the precise domains of validity and uniformity of the asymptotics, proofs of which are given next in Section~\ref{sec:OutsideDomain}.

\section{Domains of validity of the exterior asymptotic formul\ae}
\label{sec:OutsideDomain}
The exterior asymptotic formul\ae\ \eqref{eq:genus-zero-Okamoto-u-asymp}, \eqref{eq:genus-zero-Okamoto-ucirc-asymp}, \eqref{eq:genus-zero-Hermite-u-asymp}, and \eqref{eq:genus-zero-Hermite-ucirc-asymp} have been derived under the assumption that as $y$ is brought in from $|y|=\infty$ the level set $\mathrm{Re}(h(z))=0$ retains a suitable topological structure.  Since this structure is preserved as long as $\mathrm{Re}(h(\gamma))=0$ where $\gamma=U_{0,\mathrm{gO}}(y;\kappa)$ or $\gamma=U_{0,\mathrm{gH}}^{[3]}(y;\kappa)$ in the gO and gH cases respectively,  to determine the maximal domain where these results are valid we must therefore find the points $y\in\mathbb{C}$ with $0\le \arg(y)\le \tfrac{1}{2}\pi$ where $\mathrm{Re}(h(\gamma))=0$ for the indicated solution branches of the quartic $Q(\gamma,y;\kappa)=0$ defined in \eqref{eq:gamma-eqn}. Some, but not all, of these points will lie on the boundary of the region of validity of the exterior asymptotic formul\ae\ corresponding to the selected solution branch.

\subsection{Boundary curves}
Actually, we first consider the following generalization:  given $s=\pm 1$ and $\kappa\in (-1,1)$, find all points $y\in \mathbb{C}$ such that $\mathrm{Re}(h(\gamma))=0$ holds for \emph{any} solution $\gamma$ of the quartic equation \eqref{eq:gamma-eqn}.  Now at the beginning of  Section~\ref{sec:h-analysis} it was shown that for any fixed pair $(y,\gamma)$ on the Riemann surface of the equation \eqref{eq:gamma-eqn}, the function $z\mapsto \mathrm{Re}(h(z))$ is single-valued on a two-sheeted Riemann surface $\mathcal{R}$ over the $z$-plane.  Upon evaluating at $z=\gamma$ we obtain a single-valued function on a two-sheeted covering, denoted $\mathcal{Y}$, of the Riemann surface of the quartic \eqref{eq:gamma-eqn}.  

\subsubsection{Rational parametrization of the Riemann surface of \eqref{eq:gamma-eqn}}
In fact, $\mathcal{Y}$ can be identified as a two-sheeted covering of the Riemann sphere, because the quartic \eqref{eq:gamma-eqn} can be rationally parametrized, as we will now show.  Noting the symmetry $Q(-\gamma,-y;\kappa)=Q(\gamma,y;\kappa)$ of the polynomial in \eqref{eq:gamma-eqn}, we introduce the invariant quantities $p:=y\gamma$ and $q:=\gamma^2$ and hence \eqref{eq:gamma-eqn} becomes a bi-quadratic equation in $(p,q)$ which can be written in the form
\eq
\left(\frac{1}{2}q-2\kappa\right)^2-\left(p+q\right)^2=-4(1-\kappa^2)<0.
\label{eq:pq-biquadratic}
\endeq
Noting the sign of the right-hand side for $\kappa\in (-1,1)$, 
we use a rational parametrization based on stereographic projection via the identity
%it is convenient to use two different rational parametrizations depending upon whether $-1<\kappa<1$ or $\kappa>1$.  Both are based on stereographic projection:
%\begin{itemize}
%\item If $-1<\kappa<1$, we use the identity 
$(at-at^{-1})^2-(at+at^{-1})^2=-4a^2$ with $a=\sqrt{1-\kappa^2}>0$.  Hence we can identify $q$ and $p$ in \eqref{eq:pq-biquadratic} with
\eq
p=p(t)=-\sqrt{1-\kappa^2}(t-3t^{-1})-4\kappa\quad\text{and}\quad
q=q(t)=2\sqrt{1-\kappa^2}(t-t^{-1})+4\kappa.
\label{eq:pq-kappa-0}
\endeq
In particular, $y^2=p^2/q$ is explicitly given by
\eq
y^2 =y^2(t)= \frac{\sqrt{1-\kappa^2}}{2}\cdot \frac{(t^2+4wt-3)^2}{t(t^2+2wt-1)},\quad w:=\frac{\kappa}{\sqrt{1-\kappa^2}}.
\label{eq:ysquared-kappa-0}
\endeq
Given $t$, if $y^2\neq 0$ then from each choice of the square root to determine $y$ we obtain a unique corresponding value of $\gamma$ from $\gamma=p/y$.  Therefore each $t\in\mathbb{C}$ generates a symmetric pair of points $(y,\gamma)$ and $(-y,-\gamma)$ on the quartic curve $Q(\gamma,y;\kappa)=0$ in $\mathbb{C}^2$.  Note that $\kappa\mapsto w$ is a strictly increasing function of $(-1,1)$ onto $\mathbb{R}$.  Indeed, using the parametrization $\kappa=\sin(\tfrac{1}{2}\varphi)$, $\varphi\in(-\pi,\pi)$, we have simply $w=\tan(\tfrac{1}{2}\varphi)$.
%\item If $\kappa>1$, we use instead the identity $(as+as^{-1})^2-(as-as^{-1})^2=4a^2$ with $a=\sqrt{\kappa^2-1}>0$.  Therefore $p$ and $q$ are parametrized by $s\in\mathbb{C}$ as follows:
%\eq
%p=-4\kappa-\sqrt{\kappa^2-1}(s+3s^{-1})\quad\text{and}\quad q=2\sqrt{\kappa^2-1}(s+s^{-1})+4\kappa.
%\label{eq:pq-kappa-3}
%\endeq
%In this case, it is convenient to make a further fractional linear substitution $s\mapsto t$ by 
%\eq
%t=\frac{\sqrt{2}s}{\sqrt{\kappa-1}s+\sqrt{\kappa+1}},\quad\text{with inverse}\quad
%s = \frac{\sqrt{\kappa+1}t}{\sqrt{2}-\sqrt{\kappa-1}t},
%\label{eq:ts-kappa-3}
%\endeq
%where the square roots are all positive, and henceforth we write $p=p(t)$ and $q=q(t)$.  In particular, $y^2=p^2/q$ is now given by
%\eq
%y^2=y^2(t)=-\sqrt{\frac{\kappa-1}{2}}\frac{(t^2+4wt-3)^2}{t(t^2+2wt-1)},\quad w:=\frac{\kappa-3}{2\sqrt{2(\kappa-1)}}.
%\label{eq:ysquared-kappa-3}
%\endeq
%Again taking a square root to find $y$, we get $\gamma$ uniquely from $\gamma=p/y$.
%\end{itemize}

\subsubsection{Relating the condition $\mathrm{Re}(h(\gamma))=0$ to the v-trajectories of a rational quadratic differential}
Now, $\Phi(t):=-4(h(\gamma)-h(\alpha))$ is a multivalued function on $\mathcal{Y}$ due to purely real residues at its poles and the ambiguity of integration contour, but locally
it can be written as 
\eq
\Phi(t)=\int_{\alpha(t)}^{\gamma(t)}\sqrt{(z-\alpha(t))(z-\beta(t))}\frac{z-\gamma(t)}{z}\,\dd z
\label{eq:Phi-t}
\endeq
for some branch of the square root that is continuous along the unspecified path of integration, and in which $\alpha$ and $\beta$ are determined up to permutation symmetry in terms of $(y,\gamma)$ from \eqref{eq:case-iv-coefficient-match}, and $(y,\gamma)$ are in turn related by \eqref{eq:gamma-eqn}.  Thus after rational parametrization and choice of square roots in obtaining $(y,\gamma)$ from $(p=y\gamma,q=\gamma^2)$, $\alpha$, $\beta$, and $\gamma$ become functions of $t$.
Differentiation with respect to $t$ using the fact that the integrand vanishes at $z=\alpha$ and $z=\gamma$ gives
\eq
\Phi'(t)=-\frac{1}{2}\int_{\alpha(t)}^{\gamma(t)}\frac{P_2(z;t)\,\dd z}{z\sqrt{(z-\alpha(t))(z-\beta(t))}},
\label{eq:Phi-prime-general}
\endeq
where $P_2(\cdot;t)$ is the quadratic polynomial
\eq
P_2(z;t):=\left[2\gamma'(t)+\Sigma'(t)\right]z^2 - \left[\Pi'(t)+\gamma(t)\Sigma'(t)+2\Sigma(t)\gamma'(t)\right]z + \gamma(t)\Pi'(t)+2\Pi(t)\gamma'(t),
\endeq
in which $\Sigma(t):=\alpha(t)+\beta(t)$ and $\Pi(t):=\alpha(t)\beta(t)$.  Implicit differentiation of the identity $\Pi=16\gamma^{-2}$ (cf.\@ \eqref{eq:alpha-beta-E}) proves that $P_2(0;t)$ vanishes identically, so $P_2(z;t)/z$ is the linear function
\eq
P_1(z;t):=\frac{P_2(z;t)}{z}=\left[2\gamma'(t)+\Sigma'(t)\right]z - \left[\Pi'(t)+\gamma(t)\Sigma'(t)+2\Sigma(t)\gamma'(t)\right].
\endeq
Now, the derivative with respect to $z$ of $(z-\alpha(t))(z-\beta(t))$ is also a linear function of $z$, namely $2z-\Sigma(t)$.  It turns out that $P_1(z;t)$ is proportional to the latter linear function, and this makes the integrand in \eqref{eq:Phi-prime-general} the $z$-derivative of an algebraic function, allowing the integral to be evaluated explicitly.  For this, it is sufficient to check that the root of $P_1(z;t)$ agrees with that of $2z-\Sigma(t)$, i.e., that $\delta(t)\equiv 0$, where 
\eq
\delta(t):=2(\Pi'(t)+\gamma(t)\Sigma'(t)+2\Sigma(t)\gamma'(t))-\Sigma(t)(2\gamma'(t)+\Sigma'(t)).
\endeq
Eliminating $\Pi'(t)$, $\Sigma(t)$, and $\Sigma'(t)$ in favor of $\gamma(t)$ and $y(t)$ and their derivatives using \eqref{eq:alpha-beta-E} and implicit differentiation yields
\eq
\delta(t)=-64\gamma(t)^{-3}\gamma'(t)-16\gamma(t)y'(t)-16y(t)\gamma'(t)-12\gamma(t)\gamma'(t)-16y(t)y'(t) 
\endeq
so $\delta(t)=F'(t)$ where $F(t)$ may be taken to be
\eq
F(t):=32\gamma(t)^{-2}-16y(t)\gamma(t)-6\gamma(t)^2-8y(t)^2 = -\frac{6}{\gamma(t)^2}\left[\gamma(t)^4+\frac{8}{3}y(t)\gamma(t)^3+\frac{4}{3}y(t)^2\gamma(t)^2-\frac{16}{3}\right].
\endeq
Therefore, using \eqref{eq:gamma-eqn} we find that in fact $F(t)=16\kappa$, so indeed $\delta(t)=F'(t)=0$ holds.  Therefore
\eq
P_1(z;t)=\frac{P_2(z;t)}{z}=\left(\gamma'(t)+\frac{1}{2}\Sigma'(t)\right)(2z-\Sigma(t)) = -2y'(t)(2z-\Sigma(t)),
\endeq
where in the last equality we have again used implicit differentiation in \eqref{eq:alpha-beta-E}.
Using this in \eqref{eq:Phi-prime-general} gives
\eq
\begin{split}
\Phi'(t)&=y'(t)\int_{\alpha(t)}^{\gamma(t)}\frac{2z-\Sigma(t)}{\sqrt{(z-\alpha(t))(z-\beta(t))}}\,\dd z\\
& = 
2y'(t)\int_{\alpha(t)}^{\gamma(t)}\frac{\dd}{\dd z}\sqrt{(z-\alpha(t))(z-\beta(t))}\,\dd z \\ &= 
2y'(t)\sqrt{(\gamma(t)-\alpha(t))(\gamma(t)-\beta(t))}.
\end{split}
\endeq
Therefore, $\Phi'(t)^2$ is well-defined in terms of $y'(t)^2$, $y(t)$ and $\gamma(t)$ by
\eq
\begin{split}
\Phi'(t)^2&=4y'(t)^2(\gamma(t)-\alpha(t))(\gamma(t)-\beta(t))\\
&=4y'(t)^2(\gamma(t)^2-\Sigma(t)\gamma(t)+\Pi(t))\\
&=4y'(t)^2(\gamma(t)^2+(4y(t)+2\gamma(t))\gamma(t)+16\gamma(t)^{-2})\\
&=4y'(t)^2(3\gamma(t)^2+4y(t)\gamma(t)+16\gamma(t)^{-2}),
\end{split}
\endeq
where on the penultimate line we used \eqref{eq:alpha-beta-E} to eliminate $\Sigma(t)$ and $\Pi(t)$.  It happens that $\Phi'(t)^2$ is actually a rational function of $t$.  To this end, we first express it as a rational function of $p(t):=y(t)\gamma(t)$ and $q(t):=\gamma(t)^2$:
\eq
\begin{split}
\Phi'(t)^2&=4y'(t)^2(3q(t)+4p(t)+16q(t)^{-1})\\
&=(2y(t)y'(t))^2y(t)^{-2}(3q(t)+4p(t)+16q(t)^{-1})\\
&=\left(\frac{\dd}{\dd t}y(t)^2\right)^2y(t)^{-2}(3q(t)+4p(t)+16q(t)^{-1})\\
&=\left(\frac{\dd}{\dd t}\frac{p(t)^2}{q(t)}\right)^2\frac{q(t)}{p(t)^2}(3q(t)+4p(t)+16q(t)^{-1})\\
&=\frac{\left(2q(t)p'(t)-p(t)q'(t)\right)^2(3q(t)^2+4p(t)q(t)+16)}{q(t)^4}.
\end{split}
\endeq
Finally, we introduce the rational expressions \eqref{eq:pq-kappa-0} for $p$ and $q$ in terms of $t$, which yields
%which are different in the cases $-1<\kappa<1$ and $\kappa>1$:
%\begin{itemize}
%\item Using \eqref{eq:pq-kappa-0} as appropriate for $-1<\kappa<1$, we find
\eq
\Phi'(t)^2 = (1-\kappa^2)\frac{(t^4+6t^2+8wt-3)^3}{t^4(t^2+2wt-1)^4},\quad w=\frac{\kappa}{\sqrt{1-\kappa^2}}.
\label{eq:PhiPrimeSquared}
\endeq
%\item
%Using \eqref{eq:pq-kappa-3} followed by \eqref{eq:ts-kappa-3} as appropriate for $\kappa>1$, we find
%\eq
%\Phi'(t)^2=2(\kappa-1)\frac{(t^4+6t^2+8wt-3)^3}{t^4(t^2+2wt-1)^4},\quad w=\frac{\kappa-3}{2\sqrt{2(\kappa-1)}}.
%\endeq
%\end{itemize}
We therefore conclude that, since $\mathrm{Re}(h(\alpha))=0$, the curves on the $t$-sphere along which $\mathrm{Re}(h(\gamma))=0$ for any branch $\gamma$ of the quartic \eqref{eq:gamma-eqn} are v-trajectories of a rational quadratic differential $\Phi'(t)^2\,\dd t^2$ (i.e., curves in the $t$-sphere along which $\Phi'(t)^2\,\dd t^2<0$ holds).
% having the same form in both cases $-1<\kappa<1$ and $\kappa>1$ (up to an inessential positive factor that does not affect the v-trajectories) where the only distinction between these cases arises in the way that the parameter $w\in\mathbb{R}$ depends on $\kappa$.  In each case $w$ ranges over all real numbers as $\kappa$ varies in the indicated range, and $w$ is a strictly increasing function of $\kappa$.  Comparing \eqref{eq:ysquared-kappa-0} with \eqref{eq:ysquared-kappa-3} we see furthermore that given a value of $w\in\mathbb{R}$, which uniquely determines the v-trajectories on the $t$-sphere as well as two values of $\kappa$, one with $-1<\kappa<1$ and one with $\kappa>1$, the two images of these v-trajectories in the $y$-plane are related by a homothetic dilation and a rotation by $90^\circ$, both of which fix the origin.  
%
%\textcolor{red}{Remember to include this symmetry in the results section.}

\subsubsection{Critical points of $\Phi'(t)^2\,\dd t^2$ and the role of the critical v-trajectories}
\label{sec:K=LevelSet-t-plane}
To study the level curves of $\mathrm{Re}(\Phi)=0$ on the $t$-sphere, we recall the fact that $\mathrm{Re}(\Phi)$ is a single-valued function on $\mathcal{Y}$ that is non-constant and takes opposite values on the two sheets of $\mathcal{Y}$.  Moreover $\mathrm{Re}(\Phi)$ is harmonic on $\mathcal{Y}$ except at finitely many isolated singular points; hence by the same argument as in Section~\ref{sec:h-analysis} there can be no divergent v-trajectories on either $\mathcal{Y}$ or on the $t$-sphere (the latter being the projections of the former), and by the Basic Structure Theorem \cite[pg.\@ 37]{Jenkins58} the critical v-trajectories (i.e., those emanating from the zeros and simple poles (if any) of $\Phi'(t)^2$) divide the $t$-sphere into a finite union of end domains, circle domains, ring domains, and strip domains.  The immediate goal is to show that closure of the union of these critical v-trajectories is exactly the zero level set on the $t$-sphere of $\mathrm{Re}(\Phi(t))$.

First consider the zeros of $\Phi'(t)^2\,\dd t^2$, i.e., the roots of the polynomial $z(t):=t^4+6t^2+8wt-3$ that do not coincide with any zeros of the denominator.  For real $w$, $z(t)$
has real coefficients and its discriminant is proportional to $(w^2+1)^2$ which cannot vanish for any $w\in\mathbb{R}$.  For $w=0$, the roots are an opposite real pair $t=\pm\sqrt{2\sqrt{3}-3}$ a (purely imaginary) complex conjugate pair $t=\pm\ii\sqrt{2\sqrt{3}+3}$.  Since the roots must retain Schwarz symmetry and remain distinct as $w\in\mathbb{R}$ varies, and since no root can vanish for any $w$ because $z(0)<0$, this basic structure persists for all $w\in\mathbb{R}$.  We label the real roots as $a(w)<0<b(w)$ and the complex conjugate roots as $\tau(w)$ and $\tau(w)^*$ with $\mathrm{Im}(\tau(w))>0$.  Next consider the finite poles of $\Phi'(t)^2\,\dd t^2$, i.e., $t=0$ and the roots of $p(t):=t^2+2wt-1$ (since $\Phi'(t)^2$ has a finite nonzero limit as $t\to\infty$ there is also a pole at $t=\infty$ in the local coordinate $1/t$).   The roots of $p(t)$ are $t=t_\infty^\pm(w):=-w\pm\sqrt{w^2+1}$, and we compute that $z(t_\infty^\pm(w))=1+2w^2\pm 2w\sqrt{1+w^2}$.  Since $(1+2w^2)^2=(2w\sqrt{1+w^2})^2+1$ we have $z(t_\infty^\pm(w))>0$, so the poles $t=t_\infty^\pm(w)$ therefore lie outside the interval $[a(w),b(w)]$.  Moreover $t_\infty^+(w)t_\infty^-(w)=-1$, so we have the strict ordering $t^-_\infty(w)<a(w)<0<b(w)<t_\infty^+(w)$.  Since this shows that none of the four distinct roots of $z(t)$ coincides with a zero of the denominator of $\Phi'(t)^2$, these are all third-order zeros of $\Phi'(t)^2\,\dd t^2$.  Likewise, the four poles of $\Phi'(t)^2\,\dd t^2$ on the $t$-sphere are all fourth-order poles.

Upon taking a square root, we see that the four zeros and the four poles of $\Phi'(t)^2\,\dd t^2$ on the $t$-sphere are the only points of nonanalyticity of $\Phi'(t)$ and hence of any branch of $\Phi(t)$.  The four poles are clearly mapped out of the finite $y$-plane by \eqref{eq:ysquared-kappa-0},
% or \eqref{eq:ysquared-kappa-3}, 
while the four zeros are taken to finite values of $y$.  We claim that these finite values of $y$ are necessarily solutions of $B(y;\kappa)=0$ (cf.\@ \eqref{eq:branch-points}).
Indeed, if $t$ corresponds to a point on the Riemann surface of \eqref{eq:gamma-eqn} that is not a branch point (i.e., $\gamma(t)$ is not a double root of \eqref{eq:gamma-eqn} for $y=y(t)$), then the integral formula \eqref{eq:Phi-t} for $\Phi(t)$ obviously has an analytic $t$-derivative determined up to a sign (because $\gamma(t)$ is distinct from $\alpha(t)$ and $\beta(t)$ and the latter are analytic functions of $t$).  Therefore, if $z(t)=0$ making $\Phi'(t)$ nonanalytic, then $(y(t),\gamma(t))$ is a branch point of the Riemann surface of \eqref{eq:gamma-eqn} and hence $B(y(t);\kappa)=0$ for $y(t)=\pm\sqrt{y^2(t)}$ as $B(y;\kappa)$ is the discriminant of \eqref{eq:gamma-eqn}.

The condition $z(t)=0$ therefore implies that $\gamma(t)$ is a double root of \eqref{eq:gamma-eqn} for $y=y(t)$, and furthermore $\gamma(t)$ coincides with either $\alpha(t)$ or $\beta(t)$.   From the formula \eqref{eq:Phi-t} we obtain that $z(t)=0$ implies that $\mathrm{Re}(\Phi(t))=0$.  Therefore the four third-order zeros of $\Phi'(t)^2\,\dd t^2$ are all points on the zero level set of $\mathrm{Re}(\Phi(t))$, as are all points on the v-trajectories emanating from these points.  Since $\Phi'(t)^2\,\dd t^2$ has no simple poles to generate any further critical v-trajectories, the closure of the union of critical v-trajectories is contained within the zero level set of $\mathrm{Re}(\Phi(t))$.  

The complement in the $t$-sphere of the closure of the union of critical v-trajectories generally consists of finitely many disjoint end, strip, circle, and ring domains.  However there cannot be any strip or ring domains because by definition each such domain supports a single-valued branch of $\mathrm{Re}(\Phi(t))$ taking distinct values on disjoint components of its boundary, in contradiction to the assertion that $\mathrm{Re}(\Phi(t))=0$ holds unambiguously on all critical v-trajectories.  There are no circle domains either, because $\Phi'(t)^2\,\dd t^2$ has no double poles.  So all domains are end domains.  By definition, each end domain is mapped by a single-valued analytic branch of $\Phi(t)$ onto the open right or left half-plane.  Therefore there can be no interior points of any end domain with $\mathrm{Re}(\Phi(t))=0$, i.e., there are no components of the zero level set of $\mathrm{Re}(\Phi(t))$ other than the closure of the union of critical v-trajectories.

\subsubsection{Local structure of the critical v-trajectories}
Since the (two real and two complex conjugate) zeros of $\Phi'(t)^2\,\dd t^2$ are triple roots, there are five critical v-trajectories emanating from each at equal angles of $\tfrac{2}{5}\pi$.  For the real zeros we can say more:  by Schwarz symmetry there is exactly one of the five v-trajectories from each of $t=a(w)$ and $t=b(w)$ that is contained in the real line.  Since each of the four poles of $\Phi'(t)\,\dd t^2$ on the $t$-sphere is of order $4$, and since there are no strip domains in this problem, each pole has a neighborhood that is covered by the closure of the union of two disjoint end domains, and there are exactly two critical v-trajectories tending to each pole in opposite directions.  Since all four poles lie on the real equator of the $t$-sphere, by Schwarz symmetry these two critical v-trajectories are either contained in the equator or have tangents at the pole that are perpendicular to the equator.  

\subsubsection{Global structure of the critical v-trajectories}
Since exactly five critical v-trajectories emanate from each of four zeros of $\Phi'(t)^2\,\dd t^2$, the union of these consists of finitely many ($\le 20$) analytic arcs.  Since there are no divergent v-trajectories, each arc emanating from a zero terminates at a zero (the same one in a different direction, or another one) or at one of the poles (either from within the real $t$-axis/equator or perpendicular to it).  This system of arcs is symmetric under Schwarz reflection through the real $t$-axis/equator.  Note that $\Phi'(t)^2>0$ for $t<a(w)$ and $t>b(w)$ while $\Phi'(t)^2<0$ for $a(w)<t<b(w)$.  Therefore, the real intervals $(a(w),0)$ and $(0,b(w))$ are both critical v-trajectories and no other critical v-trajectories can terminate at the pole $t=0$.  Also, the critical v-trajectories that terminate at each of the three nonzero poles on the equator have tangents perpendicular to the equator.  Moreover,  fixing the positive square root of $\Phi'(t)^2$ in the interval $b(w)<t<t_\infty^+(w)$ and integrating from $t=b(w)$ one can easily see that no critical v-trajectory can cross the real axis in this interval because $\Phi(t)-\Phi(b(w))>0$.  Continuing $\Phi(t)$ around the pole at $t=t_\infty^+(w)$ one can find a point in the interval $t_\infty^+(w)<t<+\infty$ at which $\mathrm{Re}(\Phi(t))$ is finite; then by integration along the real line from this point one sees that $\mathrm{Re}(\Phi(t))$ is strictly monotone with range $\mathbb{R}$ on $t_\infty^+(w)<t<+\infty$ and therefore there is exactly one point $t_0^+(w)$ in this interval at which $\mathrm{Re}(\Phi(t))=0$ and hence a critical v-trajectory crosses the real axis exactly at this point and nowhere else in the interval.  Similarly, no critical v-trajectory can cross the real axis in the interval $t_\infty^-(w)<t<a(w)$ but there is a unique point $t_0^-(w)<t_\infty^-(w)$ such that a critical v-trajectory crosses the interval $t<t_\infty^-(w)$ at $t=t_0^-(w)$ and nowhere else.  

Now consider the critical v-trajectories emanating from zeros of $\Phi'(t)^2\,\dd t^2$ into the open upper half-plane/hemisphere.  There are two such v-trajectories emanating from each of the real zeros $t=a(w)$ and $t=b(w)$ (their Schwarz reflections enter the lower half-plane, and the remaining v-trajectory from each is real), and five emanating from $t=\tau(w)$ which lies in the open upper half-plane, for a grand total of nine arcs.  Exactly one of these arcs terminates on the boundary of the upper hemisphere at each of the three nonzero poles $t=t_\infty^-(w)$, $t=t_\infty^+(w)$, and $t=\infty$ lying on the equator.  Exactly two more arcs exit the upper hemisphere by crossing the equator at the points $t=t_0^-(w)$ and $t=t_0^+(w)$.  Therefore, there remain $9-5=4$ critical v-trajectory arcs that emanate into the upper hemisphere from one zero and terminate at the same or another zero of $\Phi'(t)^2\,\dd t^2$. An application of Lemma~\ref{lem:Teichmueller} shows that such an arc cannot originate from and return to the same zero without encircling a pole, hence without exiting the upper hemisphere since the poles lie on the equator.  Therefore the four arcs in question connect the three zeros of $\Phi'(t)^2\,\dd t^2$ in the closed upper hemisphere in pairs, i.e., there are really just two such critical v-trajectories, each with two distinct endpoints.  It is easy to see that the only possibilities are:
\begin{itemize}
\item $t=a(w)$ connected to both $t=\tau(w)$ and $t=b(w)$ with different v-trajectories;
\item $t=b(w)$ connected to both $t=a(w)$ and $t=\tau(w)$ with different v-trajectories;
\item $t=\tau(w)$ connected to both $t=a(w)$ and $t=b(w)$ with different v-trajectories.
\end{itemize}
Each of these gives rise to a Jordan curve composed of v-trajectories and their endpoints, but unfortunately applying Lemma~\ref{lem:Teichmueller} to this curve does not rule out any of these options.  

Consider then the special case $w=0$.  If $w=0$, then $\Phi'(t)^2$ is an even function of $t$ and this implies that the global structure of the critical v-trajectories is symmetric with respect to reflection through the origin $t\mapsto -t$ in addition to the Schwarz reflection through the real $t$-axis.  One then has $a(0)=-b(0)$, $t_0^-(0)=-t_0^+(0)$, and $t_\infty^-(0)=-t_\infty^+(0)$, and the existence of a v-trajectory connecting $t=\tau(0)$ with $t=a(0)$ implies the existence of a v-trajectory connecting $t=\tau(0)$ with $t=b(0)$ and vice-versa.  Hence only the third option is possible.  Lemma~\ref{lem:Teichmueller} then shows that the interior angles at the vertices $t=a(0)$, $t=\tau(0)$, and $t=b(0)$ of the critical v-trajectory ``triangle'' with real leg $(a(0),0)\cup (0,b(0))$ are all $\tfrac{2}{5}\pi$.  Since v-trajectories cannot cross at regular points, it is easy to see that the remaining v-trajectory entering the upper hemisphere from $t=b(0)$ terminates at the pole $t=t_\infty^+(0)$, the remaining v-trajectory entering the upper hemisphere from $t=a(0)$ terminates at the pole $t=t_\infty^-(0)$ (by symmetry), and two of the three remaining v-trajectories entering the upper hemisphere from $t=\tau(0)$ exit the hemisphere at the points $t=t_0^\pm(0)$ while the third terminates at the pole $t=\infty$ on the equator.  Applying Schwarz reflection to obtain the v-trajectories in the lower hemisphere, the global critical v-trajectory structure is therefore determined for the special case $w=0$.

We next show that the third option persists for all $w\in\mathbb{R}$.  To do this, we first observe that the condition that there is no critical v-trajectory connecting a given pair of zeros of $\Phi'(t)^2\,\dd t^2$, say $t=t_1(w)$ and $t=t_2(w)$ is open with respect to $w\in\mathbb{R}$.  Indeed, let $C\subset \mathbb{C}$ be a circle in the finite $t$-plane with $t_1(w)$ in its interior and $t_2(w)$ in its exterior, and assume that $C$ consists entirely of points $t$ with $\Phi'(t)^2$ finite and nonzero.  Let $A_j(w)$ denote the finite set of points on $C$ that lie on v-trajectories emanating from $t=t_j(w)$, $j=1,2$.  A v-trajectory connects $t_1(w)$ and $t_2(w)$ if and only if $A_1(w)\cap A_2(w)\neq\emptyset$.  Defining $d(w)$ as the minimum arc length distance along $C$ between points of $A_1(w)$ and $A_2(w)$ (and taking $d(w)=+\infty$ if either $A_1(w)$ or $A_2(w)$ is empty), we can see that $d(w)\ge 0$ and $d(w)=0$ if and only if $A_1(w)\cap A_2(w)\neq\emptyset$.  One can also show that $d(w)$ is lower semicontinuous on $\mathbb{R}$, and therefore the inverse image of $d>0$ is open, which proves the observation.  Now consider the open set $w\in U\subset\mathbb{R}$ for which $a(w)$ and $b(w)$ are not connected by a v-trajectory of $\Phi'(t)^2\,\dd t^2$.  Consider also the closed set $w\in F\subset\mathbb{R}$ for which there is a v-trajectory connecting $a(w)$ with $\tau(w)$ and another v-trajectory connecting $b(w)$ with $\tau(w)$.  An examination of the three options above shows that $F=U$, as both conditions correspond to the third option.  Moreover $w=0$ obviously belongs to $F=U$, so the latter is nonempty.  Since $\mathbb{R}$ is connected, it then follows that $F=U=\mathbb{R}$.  

The global structure of critical v-trajectories for $\Phi'(t)^2\,\dd t^2$ is illustrated for two values of $w\in\mathbb{R}$ in Figure~\ref{fig:t-plane-trajectories}.
\begin{figure}[h]
\begin{center}
\includegraphics{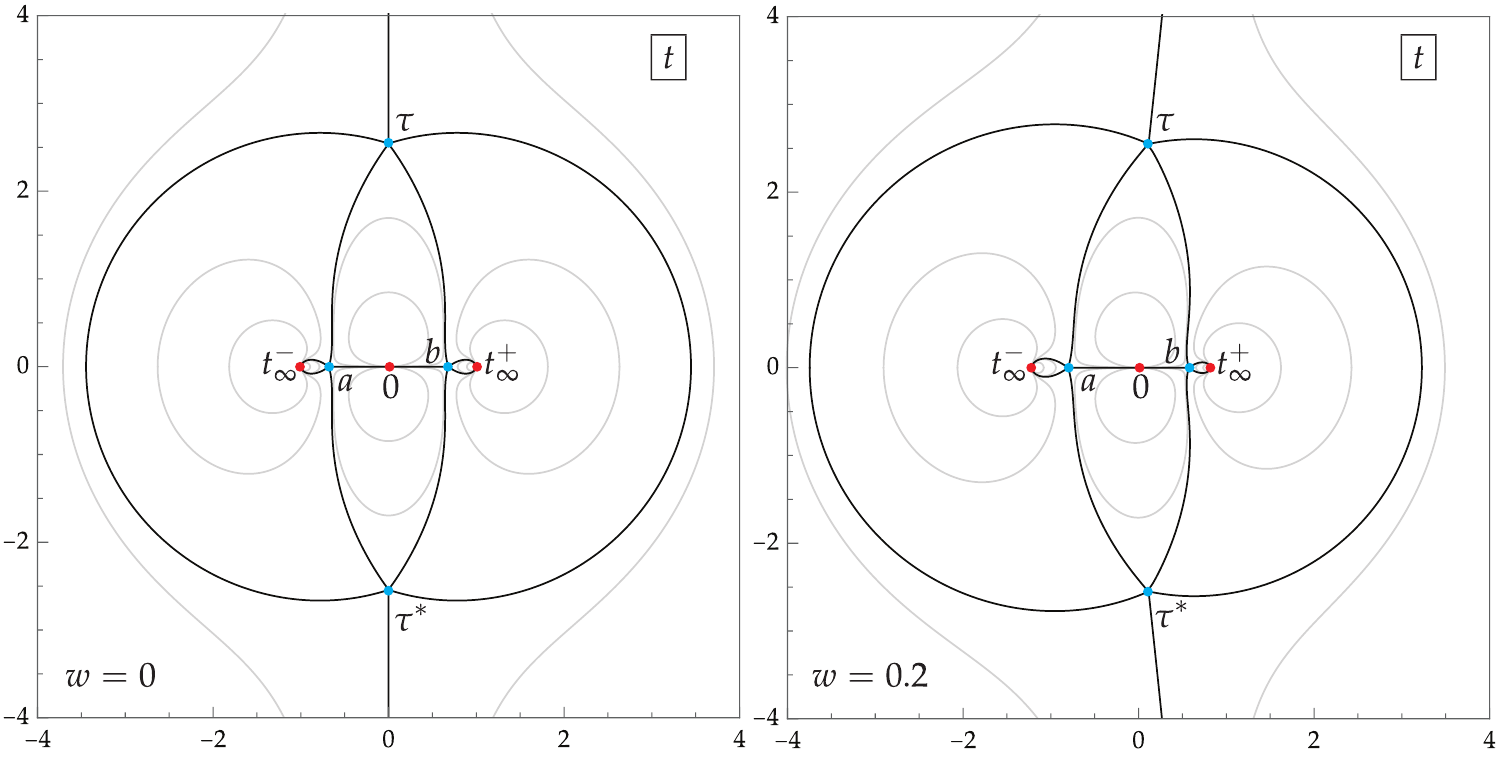}
\end{center}
\caption{The critical v-trajectories (black) and some noncritical v-trajectories (gray) of $\Phi'(t)^2\,\dd t^2$ (numerically generated) for $w=0$ (left panel) and $w=0.2$ (right panel).  Also shown are the zeros (cyan) and finite poles (red) of $\Phi'(t)^2$.}
\label{fig:t-plane-trajectories}
\end{figure}

\subsubsection{Image in the $y$-plane}
Recall that under \eqref{eq:ysquared-kappa-0} 
%or \eqref{eq:ysquared-kappa-3} 
the points on the Riemann surface of \eqref{eq:gamma-eqn} that are mapped to $y=\infty$ correspond exactly to $t=t_\infty^\pm(w)$, $t=0$, and $t=\infty$.  It is easy to see by expanding $y^2$ and $q=\gamma^2$ for small $t$ that it is the point $t=0$ that is mapped to the point at $y=\infty$ on the branch $\gamma=U_{0,\mathrm{gO}}(y;\kappa)$ of \eqref{eq:gamma-eqn} for which $\gamma=-\tfrac{2}{3}y + \bo(1)$ as $y\to\infty$.   Every neighborhood of $t=0$ contains parts of the two vertical critical v-trajectories coinciding with the real intervals $a(w)<t<0$ and $0<t<b(w)$, and if the neighborhood is sufficiently small these are the only critical v-trajectories contained.  Recall that $t_\infty^-(w)<a(w)<0<b(w)<t_\infty^+(w)$ and that $a(w)$ and $b(w)$ are mapped by $y=\pm\sqrt{y^2(t)}$ to solutions of $B(y;\kappa)=0$ none of which can vanish.  Then, from \eqref{eq:ysquared-kappa-0} one sees that 
%for $-1<\kappa<1$ 
the real v-trajectory $a(w)<t<0$ is mapped bijectively onto $0<y^2(a(w))<y^2<+\infty$ while the real v-trajectory $0<t<b(w)$ is mapped bijectively onto $-\infty<y^2<y^2(b(w))<0$. 
%Likewise, using \eqref{eq:ysquared-kappa-3} one sees that for $\kappa>1$ the real v-trajectory $a(w)<t<0$ is mapped bijectively onto $-\infty<y^2<y^2(a(w))<0$ while the real v-trajectory $0<t<b(w)$ is mapped bijectively onto $0<y^2(b(w))<y^2<+\infty$.  
Upon taking square roots and recalling the interpretation of critical v-trajectories, we find that $\mathrm{Re}(h(U_{0,\mathrm{gO}}(y;\kappa)))=0$ holds on the four coordinate axes in the $y$-plane outside of the four purely real and purely imaginary solutions of $B(y;\kappa)=0$.  Similarly, by expanding $y^2$ and $q=\gamma^2$ for large $t$ one sees that $t=\infty$ is mapped to $y=\infty$ on the branch $\gamma=U_{0,\mathrm{gH}}^{[3]}(y;\kappa)$ of \eqref{eq:gamma-eqn} for which $\gamma=-2y+\bo(1)$ as $y\to\infty$, and that the images of the two critical v-trajectories tending to $t=\infty$ parallel to the direction $\arg(t)=\pm\tfrac{1}{2}\pi$ are four unbounded curves in the $y$-plane with asymptotic arguments $\arg(y)=\pm\tfrac{1}{4}\pi,\pm\tfrac{3}{4}\pi$ along each of which $\mathrm{Re}(h(U_{0,\mathrm{gH}}^{[3]}(y;\kappa)))=0$.

Let us refer to the interior of the closure of the union of the two end domains abutting the pole at $t=0$ (resp., $t=t_\infty^\pm$, $t=\infty$) as $S_0$ (resp., $S_\pm$, $S_\infty$).  If we denote by $\partial S$ the closure in $\overline{\mathbb{C}}$ of the union of critical v-trajectories both of whose endpoints are zeros of $\Phi'(t)^2$, then the $t$-sphere is the disjoint union $S_0\sqcup S_+\sqcup S_-\sqcup S_\infty\sqcup \partial S$.  Each of the $S_j$ is a domain containing precisely one point that is mapped to $y^2=\infty$, and it is easy to see from \eqref{eq:ysquared-kappa-0} 
% and \eqref{eq:ysquared-kappa-3} 
that the mapping $t\mapsto y^2$ is locally univalent near each of these singularities.  The two domains $S_0$ and $S_\infty$ are special in that $t\mapsto y^2$ is univalent on each.  The images of these domains in the $y$-plane correspond (as it turns out) to the exterior asymptotic regions for the gO and gH Painlev\'e-IV rational solutions respectively, so we may also write $S_0=S_\mathrm{gO}$ and $S_\infty=S_\mathrm{gH}$.  The univalence of $t\mapsto y^2$ on $S_\mathrm{gO}$ and $S_\mathrm{gH}$ will be proved in Section~\ref{sec:ExteriorOkamotoUnivalence} below.  

The critical points of the mapping $t\mapsto y^2$ are easily seen to be the four zeros of $\Phi'(t)^2$ as well as the two double roots of $y^2(t)$, which as solutions of $t^2+4wt-3=0$ are clearly real for all $w\in\mathbb{R}$.  All of the critical points are simple zeros of $\dd y^2/\dd t$.  Since the resultant of $t^2+4wt-3$ and $t^2+2wt-1$ is proportional to $w^2+1$ which does not vanish for any $w\in\mathbb{R}$, by checking the relationship between the double roots of $y^2(t)$ and the poles $t_\infty^\pm(w)$ for $w=0$ one determines that for all $w\in\mathbb{R}$ the roots of $y^2$ lie one on each side of the interval $[t_\infty^-(w),t_\infty^+(w)]$.  Therefore the only critical points of $y^2(t)$ that are not in $\partial S$ are one each within $S_\pm$.  Hence it is impossible for $t\mapsto y^2$ to be univalent on $S_+$ or $S_-$.  

\subsection{Exterior asymptotics for $y$ on the images of critical v-trajectories within $S_\mathrm{gO}$ or $S_\mathrm{gH}$}
\label{sec:ExteriorAxes}
The only points in the image of $S_\mathrm{gO}$ in the $y$-plane where $\mathrm{Re}(h(U_{0,\mathrm{gO}}(y;\kappa)))=0$ holds lie on the real and imaginary axes.  Here we explain why the gO exterior asymptotic formul\ae\ \eqref{eq:genus-zero-Okamoto-u-asymp} and $\eqref{eq:genus-zero-Okamoto-ucirc-asymp}$ are still valid for $\arg(y)=0$ and $\arg(y)=\tfrac{1}{2}\pi$ for sufficiently large $|y|$, even though the level set topology for $\mathrm{Re}(h(z))=0$ has to differ from the generic case in which $0<\arg(y)<\tfrac{1}{2}\pi$.   

The idea is that ``leftward'' and ``downward'' configurations are both simultaneously valid for asymptotic analysis as $T\to +\infty$ for any $y$ in the image of $S_\mathrm{gO}$ with $0<\arg(y)<\tfrac{1}{2}\pi$.  However, only the ``leftward'' configuration allows us to study $\arg(y)=0$ while only the ``downward'' configuration allows us to study $\arg(y)=\tfrac{1}{2}\pi$.  

The asymptotic behavior of $\alpha$, $\beta$, and $\gamma$ for large $y$ described in the first lines of Section~\ref{sec:Exterior} shows that as $\arg(y)\downarrow 0$, $\alpha$ and $\gamma$ approach the negative real $z$-axis from below while $\beta$ approaches the negative real axis from above.  Roughly speaking, the connected component of $\critclosure $ containing $\alpha$ rotates around the origin in the clockwise direction while the component containing $\beta$ rotates around the origin in the counterclockwise direction, and these trap the component containing $\gamma$; in the limit all three components become one (and $\gamma$ is then on the same level as $\alpha$ and $\beta$).  In the limit, the level curve emanating from $z=\alpha$ and escaping to $\infty$ in the direction $\arg(z)=-\tfrac{1}{4}\pi$ merges at the point $z=\gamma$, which now lies on the real axis between $\alpha$ and $\beta$, with the level curve emanating from $z=\beta$ and escaping to $\infty$ in the direction $\arg(z)=\tfrac{1}{4}\pi$.  This limit process can be seen in Figure~\ref{fig:Trajectories-gO-k0} as one follows the plots down the right-hand side toward the lower right-hand corner.  In the ``leftward'' configuration shown in the left-hand panel of Figure~\ref{fig:GenusZero-y-Diag-N-to-O} the only arc of the jump contour for $\mathbf{O}(z)$ that becomes constrained by this change in level set topology as $\arg(y)\downarrow 0$ is $\Sigma_\mathrm{c}$; however this arc carries a constant (in $z$) jump condition \eqref{eq:Sigma-c-left} whose analytical properties in the limit $T\to +\infty$ do not depend on the sign chart for $\mathrm{Re}(h(z))$.  By contrast, in the ``downward'' configuration shown in the right-hand panel of Figure~\ref{fig:GenusZero-y-Diag-N-to-O}, the arcs $\Sigma_{2,3}$ and $\Sigma_{2,1}$ are additional casualties of the change in level set topology as $\arg(y)\downarrow 0$, and unlike $\Sigma_\mathrm{c}$, these contours carry jump conditions with exponentials whose rapid decay to zero as $T\to +\infty$ is ruined in the limit, near $z=\gamma$ at least.

In a similar way, one can see that as $\arg(y)\uparrow\tfrac{1}{2}\pi$ it is the ``downward'' configuration that preserves the required exponential decay in all jump matrices in the limit. The corresponding degeneration can be observed by following the plots along the top edge of Figure~\ref{fig:Trajectories-gO-k0} toward the upper left-hand corner. We therefore conclude that the asymptotic formul\ae\ \eqref{eq:genus-zero-Okamoto-u-asymp} and \eqref{eq:genus-zero-Okamoto-ucirc-asymp} have been established for all $y$ in the image of $S_\mathrm{gO}$ with $0\le\arg(y)\le\tfrac{1}{2}\pi$ (in the case of \eqref{eq:genus-zero-Okamoto-ucirc-asymp} $y$ has been rescaled by positive factor depending on $s$ and $\kappa$).  Since $u$, $u_\tw$, and $U_{0,\mathrm{gO}}$ are all odd Schwarz-symmetric functions of arguments related by positive scaling factors, it then follows that the gO exterior asymptotic formul\ae\ 
\eqref{eq:genus-zero-Okamoto-u-asymp} and \eqref{eq:genus-zero-Okamoto-ucirc-asymp} are both valid for all $y$ in the image of $S_\mathrm{gO}$.
%$u^{(m,n)}(T^{1/2}y)$ and $\gamma=\gamma^{(\kappa)}(y)\sim -\tfrac{2}{3}y$ are odd Schwarz-symmetric functions of $y$ it then follows that \eqref{eq:exterior-asymptotic} holds for all $y$ in the image of $S_\mathrm{gO}$, which motivates the following definition.

Turning now to the gH exterior asymptotic formul\ae\ \eqref{eq:genus-zero-Hermite-u-asymp} and \eqref{eq:genus-zero-Hermite-ucirc-asymp} for $y$ in the image of $S_\mathrm{gH}$ with $0\le\arg(y)\le\tfrac{1}{2}\pi$, the only points for which $\mathrm{Re}(h(U_{0,\mathrm{gH}}^{[3]}(y;\kappa)))=0$ lie along a single unbounded curve joining the unique root of $B(y;\kappa)=0$ in the open first quadrant with $y=\infty$ in the asymptotic direction $\arg(y)=\tfrac{1}{4}\pi$.  It is an image of the critical v-trajectory connecting $t=\tau(w)$ with $t=\infty$.  Appealing to the special case that $\kappa=0$, one can see that the root of $B(y;\kappa)=0$ of interest is on the line $\arg(y)=\tfrac{1}{4}\pi$ and moreover the image curve is the straight ray with $\arg(y)=\tfrac{1}{4}\pi$ that emanates from that root.  Back in the $z$-plane, one sees that $\gamma$ is connected to $\beta$ by a v-trajectory of $Q(z)\,\dd z^2$.  Referring to Figure~\ref{fig:GenusZeroGH-Stokes} showing the critical v-trajectories for points on opposite sides of the transitional ray, the condition $\mathrm{Re}(h(U_{0,\mathrm{gH}}^{[3]}(y;\kappa)))=0$ indicates the transit of $\gamma$ through the unbounded critical v-trajectory emanating from $\beta$, upon which event the domain $\mathrm{Re}(h(z))>0$ breaks into two disjoint components while the complementary domain $\mathrm{Re}(h(z))<0$ becomes connected (or vice-versa, depending on the direction of transit).  This change in the level curve $\mathrm{Re}(h(z))=0$ does not affect the topological sign structure of $\mathrm{Re}(h(z))$ in a neighborhood of the circle domain $\circledomain$, where the jump contour for $\mathbf{O}(z)$ lies as illustrated in 
%the right-hand panel of Figure~\ref{fig:GenusZeroGH-second-two}.  
the panels of Figure~\ref{fig:GenusZeroGH-Lenses}.  
Therefore the condition $\mathrm{Re}(h(U_{0,\mathrm{gH}}^{[3]}(y;\kappa)))=0$ is a ``false alarm'' whenever it occurs on a point in the image of $S_\mathrm{gH}$ and we conclude that the asymptotic formul\ae\ \eqref{eq:genus-zero-Hermite-u-asymp} and \eqref{eq:genus-zero-Hermite-ucirc-asymp} are valid for all such $y$.  This interesting but irrelevant (for our purposes) phenomenon can be seen in the plots near the upper right-hand corner of Figure~\ref{fig:Trajectories-gH-k0}.

This motivates the following definition, which takes into account the rescaling of $y$ needed to write the asymptotic formul\ae\ \eqref{eq:genus-zero-Okamoto-ucirc-asymp} and \eqref{eq:genus-zero-Hermite-ucirc-asymp}.
\begin{definition}[Exterior domains for the Painlev\'e-IV rationals]
If $\kappa\in (-1,1)$, $\OkamotoExterior(\kappa)$ and $\HermiteExterior(\kappa)$ are the images in the $y$-plane, under \eqref{eq:ysquared-kappa-0} 
% (for $\kappa\in (-1,1)$) or \eqref{eq:ysquared-kappa-3} (for $\kappa>1$) 
followed by a double-valued square root, of the domains $S_0=S_\mathrm{gO}$ and $S_\infty=S_\mathrm{gH}$ respectively.  If instead $\pm\kappa>1$, then $\OkamotoExterior(\kappa)$ and $\HermiteExterior(\kappa)$ are defined by homothetic dilation of the definition on $(-1,1)$:
\eq
\OkamotoExterior(\kappa):=\sqrt{\frac{1\pm\kappa}{2}}\OkamotoExterior(I^\pm(\kappa))\quad\text{and}\quad
\HermiteExterior(\kappa):=\sqrt{\frac{1\pm\kappa}{2}}\HermiteExterior(I^\pm(\kappa)), \quad I^\pm(\kappa):=-\frac{\kappa\mp 3}{1\pm\kappa},
\quad\pm\kappa>1.
\label{eq:DomainExtend}
\endeq
Here the M\"obius transformations $I^\pm$ are both involutions ($I^\pm(I^\pm(\kappa))=\kappa$) and $I^+:(1,+\infty)\to(-1,1)$ 
 while $I^-:(-\infty,-1)\to (-1,1)$.
\label{def:ExteriorDomains}
\end{definition}
The asymptotic formul\ae\ \eqref{eq:genus-zero-Okamoto-u-asymp} and \eqref{eq:genus-zero-Hermite-u-asymp} are valid for each $y$ in $\OkamotoExterior(\kappa)$ and $\HermiteExterior(\kappa)$ respectively, where $\kappa\in (-1,1)$.  Likewise \eqref{eq:genus-zero-Okamoto-ucirc-asymp} and \eqref{eq:genus-zero-Hermite-ucirc-asymp}
hold for each $y$ in $\OkamotoExterior(\kappa_\tw)$ and $\HermiteExterior(\kappa_\tw)$ respectively, where $\kappa_\tw<-1$ if $s=1$ and $\kappa_\tw>1$ if $s=-1$.  In Section~\ref{sec:Exterior-Uniformity} below we indicate how the error estimates can be strengthened to obtain uniform convergence on closed subsets.  But first we show that $\OkamotoExterior(\kappa)$ and $\HermiteExterior(\kappa)$ are actually domains.

\subsection{Univalence of $t\mapsto y^2$ on $S_\mathrm{gO}$ and $S_\mathrm{gH}$ and properties of $\OkamotoExterior(\kappa)$ and $\HermiteExterior(\kappa)$}
\label{sec:ExteriorOkamotoUnivalence}
The validity of the gO rational asymptotic formula \eqref{eq:genus-zero-Okamoto-u-asymp} whenever $y\in\OkamotoExterior(\kappa)$ implies that the map $t\mapsto y^2$ given by \eqref{eq:ysquared-kappa-0}
% for $-1<\kappa<1$ or by \eqref{eq:ysquared-kappa-3} for $\kappa>1$ 
is univalent on the domain $S_\mathrm{gO}$.  Likewise, the validity of \eqref{eq:genus-zero-Hermite-u-asymp} whenever $y\in\HermiteExterior(\kappa)$ implies univalence of the same map on $S_\mathrm{gH}$.  Indeed, different values of $t$ for the same value of $y^2$ parametrize different points on the Riemann surface of the quartic \eqref{eq:gamma-eqn} over the same point in the $y^2$-plane, and hence distinct values of $\gamma$.  But neither asymptotic formula \eqref{eq:genus-zero-Okamoto-u-asymp} nor \eqref{eq:genus-zero-Hermite-u-asymp} can hold at the same value of $y$ for two distinct values of $\gamma$.  Hence there is only one value of $t\in S_\mathrm{gO}$ corresponding to each point $y$ in its image under \eqref{eq:ysquared-kappa-0}
%or \eqref{eq:ysquared-kappa-3} 
followed by a branch of the square root, and the same is true for $S_\mathrm{gH}$.  Since $S_\mathrm{gO}$ and $S_\mathrm{gH}$ are both simply connected domains, this proves that $\OkamotoExterior(\kappa)$ and $\HermiteExterior(\kappa)$ are also simply connected domains for each $\kappa\in (-1,1)$, and hence by dilation for all $\kappa\in\mathbb{R}\setminus\{-1,1\}$.  

In fact, the images of both boundaries $\partial S_\mathrm{gO}$ and $\partial S_\mathrm{gH}$ are Jordan curves.  Indeed, by univalence on either one of the domains $S_\mathrm{gO}$ or $S_\mathrm{gH}$ it follows that the image of the boundary cannot intersect itself transversely.  However there also cannot be any osculation points.  This can be proved in a similar way, using inconsistency for distinct values of $t$ mapped to the same value of $y$ of asymptotic formul\ae\ generalizing  \eqref{eq:genus-zero-Okamoto-u-asymp} and \eqref{eq:genus-zero-Hermite-u-asymp} to hold near boundary points of $\OkamotoExterior(\kappa)$ and $\HermiteExterior(\kappa)$ respectively.
The leading order terms in such generalized formul\ae\ take the form of a trigonometric perturbation of the ``background'' value given by \eqref{eq:genus-zero-Okamoto-u-asymp} and \eqref{eq:genus-zero-Hermite-u-asymp} respectively and hence determined by $t$.  We do not describe here this kind of ``edge'' asymptotic for the Painlev\'e-IV rational solutions, but the starting point is still the Riemann-Hilbert characterization of these solutions via Theorem~\ref{thm:gO-RHP} (for the gO case) or Theorem~\ref{thm:gH-RHP} (for the gH case).  The sort of analysis required has been carried out in detail for the rational Painlev\'e-II solutions; see \cite[Section 3]{BuckinghamM15}.
\label{page:Jordan}

The most important qualitative properties of $\partial\OkamotoExterior(\kappa)$ and $\partial\HermiteExterior(\kappa)$ are given in Proposition~\ref{prop:JordanCurves}, whose proof we now give.
\begin{proof}[Proof of Proposition~\ref{prop:JordanCurves}]
That $\partial\OkamotoExterior(\kappa)$ and $\partial\HermiteExterior(\kappa)$ are Jordan curves if $\kappa\in (-1,1)$ has been proved above.  Schwarz reflection symmetry of the preimage curves $\partial S_\mathrm{gO}$ and $\partial S_\mathrm{gH}$ is preserved under the map $t\mapsto y^2$ in \eqref{eq:ysquared-kappa-0}; then the square root map $y^2\mapsto y$ yields Schwarz symmetry in both real and imaginary axes for $\partial\OkamotoExterior(\kappa)$ and $\partial\HermiteExterior(\kappa)$, which proves the Schwarz reflection symmetries of these curves for $\kappa\in (-1,1)$ according to Definition~\ref{def:ExteriorDomains}.  Since the images of $t=a,b$ are the four solutions of $B(y;\kappa)=0$ on the real and imaginary axes while the images of $t=\tau,\tau^*$ are the four remaining solutions, the statements concerning the ``vertices'' of $\partial\OkamotoExterior(\kappa)$ and $\partial\HermiteExterior(\kappa)$ are proved for $\kappa\in (-1,1)$ as well.   To extend the above results to $\kappa>1$ or $\kappa<-1$, we just use Definition~\ref{def:ExteriorDomains} to relate these to curves with $\kappa\in (-1,1)$ by homothetic dilation (and use an easily verified dilation symmetry of the branch point equation $B(y;\kappa)=0$).    

The symmetries $\partial\OkamotoExterior(-\kappa)=\ii\partial\OkamotoExterior(\kappa)$ and $\partial\HermiteExterior(-\kappa)=\ii\partial\HermiteExterior(\kappa)$ for $\kappa\in (-1,1)$ follow from Definition~\ref{def:ExteriorDomains} because $\kappa\mapsto -\kappa$ is equivalent to $w\mapsto -w$.  Thus one sees from \eqref{eq:PhiPrimeSquared} that the union of v-trajectories of $\Phi'(t)^2\,\dd t^2$ is invariant under $(\kappa,t)\mapsto (-\kappa,-t)$, so $S_\mathrm{gO}$ and $S_\mathrm{gH}$ are reflected through the origin in the $t$-plane by $\kappa\mapsto -\kappa$.  But from \eqref{eq:ysquared-kappa-0} one can see that $y^2\mapsto -y^2$ when $(\kappa,t)\mapsto (-\kappa,-t)$ which proves the claim.  The remaining symmetries in \eqref{eq:geometric-symmetries} are direct consequences of Definition~\ref{def:ExteriorDomains}; these identities are simply the definition \eqref{eq:DomainExtend} rewritten using the involutive property $I^\pm(I^\pm(\kappa))=\kappa$.  The statement that the curves $\partial\OkamotoExterior(\kappa)$ and $\partial\HermiteExterior(\kappa)$ are invariant under rotation by $\tfrac{1}{2}\pi$ for $\kappa=-3,0,3$ then follows from the first symmetry in \eqref{eq:geometric-symmetries} for $\kappa=0$ and the fact that $I^\pm(0)=\pm 3$.

Finally, the analyticity of the indicated branches of the equilibrium equation \eqref{eq:equilibrium} on $\OkamotoExterior(\kappa)$ and $\HermiteExterior(\kappa)$ can be seen for $\kappa\in (-1,1)$ from the univalence of $t\mapsto y$ on $S_\mathrm{gO}$ and $S_\mathrm{gH}$ respectively.  For the corresponding results on $\kappa>1$ or $\kappa<-1$, we use Lemma~\ref{lem:gamma-identity}.
\end{proof}

\subsection{Uniformity of estimates}
\label{sec:Exterior-Uniformity}
So far, the asymptotic formul\ae\ in Theorems~\ref{thm:OkamotoExterior} and \ref{thm:HermiteExterior} have been proved to hold on the indicated domains in the sense of pointwise convergence for $y=y_0$ fixed.  However, it is a minor technical matter to strengthen the convergence to uniformity on closed subsets of the indicated (open) domains.  First suppose $y_0$ lies in a compact subset.  Then as $y_0$ varies in this subset the jump contours for the Riemann-Hilbert problem characterizing the error $\mathbf{E}(z)$ deform continuously, but by a further deformation argument one shows that one can take the jump contour to be locally independent of $y_0$.  By extracting a finite subcover of the compact set of $y_0$ one uses estimates for Cauchy integral operators on finitely many contours within the context of the small-norm theory to prove uniform convergence.  For uniformity as $y=y_0\to\infty$, one has to first perform a rescaling of the $z$-plane to fix the roots $z=\alpha$ and $z=\beta$ in the limit; however since the Cauchy integral operators commute with scaling it is then easy to adapt the uniform convergence argument to a neighborhood of $y_0=\infty$.  For details of these arguments in a similar context, see
\cite[pgs.\@ 2519--2520]{BuckinghamM:2014}.

This completes the proofs of Theorems~\ref{thm:HermiteExterior} and \ref{thm:OkamotoExterior}.

\section{Boutroux spectral curves and the Boutroux domains $\rectangle$, $\TR$, and $\TI$}
\label{sec:Boutroux}
We now begin our study of Riemann-Hilbert Problem~\ref{rhp:general} for $y$ in the bounded region complementary to $\OkamotoExterior(\kappa)$ or $\HermiteExterior(\kappa)$.  According to Theorems~\ref{thm:HermiteExterior} and \ref{thm:OkamotoExterior}, this bounded region must contain all of the poles and zeros of the rational solution in question when $T$ is large, because the equilibria characterized by \eqref{eq:equilibrium} are analytic and nonvanishing on the relevant unbounded domain $\OkamotoExterior(\kappa)$ or $\HermiteExterior(\kappa)$.  Therefore we expect non-equilibrium solutions of the approximating differential equation \eqref{eq:Approximating-ODE-Second-Order} may play a role in the asymptotic behavior and hence we replace $y$ with $y=y_0+T^{-1}\zeta$ for Sections~\ref{sec:Boutroux} and \ref{sec:G1}.  The polynomial $P(\cdot)$ defined in \eqref{eq:elliptic-ODE}, which also appears in the theory of spectral curves as explained in Section~\ref{sec:SpectralCurve}, depends parametrically on $y_0$, $\kappa$, and the integration constant $E$.  The purpose of this section is to describe in detail how $E$ is determined. 

Recalling the taxonomy of the quartic $P(z)$ presented in Section~\ref{sec:SpectralCurve}, we assume that $P(z)$ is in the most general case $\{1111\}$ where $y$ is replaced with $y_0$, having four distinct roots.  The rational equation $v^2=h'(z)^2=P(z)/(16z^2)$ then defines an elliptic curve $\mathcal{R}$ (genus $1$) parametrized by $y_0\in\mathbb{C}$, $\kappa\in(-1,1)$, and $E\in\mathbb{C}$.  We will be interested in curves satisfying the \emph{Boutroux conditions} (cf.\@ \eqref{eq:intro-Boutroux})
\eq
\mathrm{Re}\left(\oint_\mathfrak{a}v\,\dd z\right)=0\quad\text{and}\quad\mathrm{Re}\left(\oint_\mathfrak{b}v\,\dd z\right)=0
\label{eq:Boutroux}
\endeq
where $\mathfrak{a}$ and $\mathfrak{b}$ are representatives of a canonical basis of homology cycles on $\mathcal{R}$ chosen not to pass through any poles of $v$ (at the two points each over $z=0,\infty$).  Note that although the differential $v\,\dd z$ has residues at these four points, they are all purely real ($\pm 1$ for the points over $z=0$ and $\pm \kappa$ for the points over $z=\infty$) so the Boutroux conditions depend only on the homology classes of the cycles $\mathfrak{a}$ and $\mathfrak{b}$.   Moreover, since $(\mathfrak{a},\mathfrak{b})$ is a basis for the homology group of $\mathcal{R}$, the conditions \eqref{eq:Boutroux} taken together depend intrinsically on $\mathcal{R}$ itself.  We will determine $E$ as a function of $y_0$ and $\kappa$ such that $\mathcal{R}$ satisfies \eqref{eq:Boutroux}, making it a \emph{Boutroux curve}.  
%Without loss of generality, we assume that either $\kappa\in (-1,1)$ or $\kappa>1$ holds.

\subsection{Boutroux domains}
\label{sec:BoutrouxDomains}
Letting $E_\mathrm{R}:=\mathrm{Re}(E)$ and $E_\mathrm{I}:=\mathrm{Im}(E)$, and defining 
\eq
f_{\mathfrak{a},\mathfrak{b}}(E_\mathrm{R},E_\mathrm{I};y_0,\kappa):=\mathrm{Re}\left(\oint_{\mathfrak{a},\mathfrak{b}}v\,\dd z\right),
\endeq
the Boutroux conditions \eqref{eq:Boutroux} read $f_\mathfrak{a}(E_\mathrm{R},E_\mathrm{I};y_0,\kappa)=f_\mathfrak{b}(E_\mathrm{R},E_\mathrm{I};y_0,\kappa)=0$.  A direct calculation then shows that 
\eq
\frac{\partial f_{\mathfrak{a},\mathfrak{b}}}{\partial E_\mathrm{R}}=\frac{1}{4}\mathrm{Re}\left(\oint_{\mathfrak{a},\mathfrak{b}}\frac{\dd z}{4zv}\right)\quad\text{and}\quad
\frac{\partial f_{\mathfrak{a},\mathfrak{b}}}{\partial E_\mathrm{I}}=-\frac{1}{4}\mathrm{Im}\left(\oint_{\mathfrak{a},\mathfrak{b}}\frac{\dd z}{4zv}\right).
\endeq
It follows that the Jacobian of $(f_\mathfrak{a},f_\mathfrak{b})$ with respect to $(E_\mathrm{R},E_\mathrm{I})$ is
\eq
\frac{\partial (f_\mathfrak{a},f_\mathfrak{b})}{\partial(E_\mathrm{R},E_\mathrm{I})}=\frac{1}{16}\mathrm{Im}(Z_\mathfrak{a}Z_\mathfrak{b}^*),\quad Z_\mathfrak{a}:=\oint_\mathfrak{a}\frac{\dd z}{4zv}\quad Z_\mathfrak{b}:=\oint_\mathfrak{b}\frac{\dd z}{4zv}.
\endeq
Since $\dd z/(4zv)$ is up to scaling the unique nonzero holomorphic differential on the elliptic curve $\mathcal{R}$, it follows from \cite[Corollary 1]{Dubrovin81} that $\partial (f_\mathfrak{a},f_\mathfrak{b})/\partial(E_\mathrm{R},E_\mathrm{I})$ is finite and nonzero whenever $\mathcal{R}$ is a smooth elliptic curve, i.e., whenever the quartic polynomial $P(z)$ is in case $\{1111\}$, having four distinct roots.  Therefore, starting from a known solution $(E_\mathrm{R},E_\mathrm{I})=(E_\mathrm{R}^\#,E_\mathrm{I}^\#)$ of the Boutroux equations \eqref{eq:Boutroux} for $y_0=y_0^\#$ and $-1<\kappa<1$ 
%or $\kappa>1$ 
such that $\mathcal{R}$ is a nonsingular curve, the Implicit Function Theorem guarantees that we may uniquely continue this solution to $y_0\neq y_0^\#$ until $P(z)$ degenerates from case $\{1111\}$ to either case $\{31\}$ or $\{211\}$.  

Given a value of 
%$\kappa$ with 
$\kappa\in (-1,1)$,
% or $\kappa>1$, 
a maximal simply connected domain $\mathcal{B}=\mathcal{B}(\kappa)$ in the $y_0$-plane together with a smooth solution $E:\mathcal{B}\to\mathbb{R}^2$, $y_0\mapsto (E_\mathrm{R},E_\mathrm{I})$ of the Boutroux conditions \eqref{eq:Boutroux} for which $\mathcal{R}=\mathcal{R}(y_0)$ is a smooth elliptic curve (i.e., of class $\{1111\}$) will be called a \emph{Boutroux domain}.  Sometimes we will abuse terminology and refer to just the set $\mathcal{B}\subset\mathbb{C}$ as a Boutroux domain.  Note that in principle it is possible in applications of the Implicit Function Theorem to encounter isolated branch points in the $y_0$-plane from which an artificial branch cut must emanate to separate distinct branches of the solution; at nonterminal points of this branch cut the Jacobian is finite and nonzero (which is why the cut is artificial).  However, such phenomena cannot occur in the present context because different values of $E$ for the same value of $y_0$ would lead to inconsistent asymptotic behavior for the 
%generalized Okamoto 
rational Painlev\'e-IV solutions.  Thus the conditions of the Implicit Function Theorem fail at every point of the boundary of a Boutroux domain $\mathcal{B}$, which justifies the assertion of uniqueness of the maximal domain.  

\subsection{Stokes graphs and abstract Stokes graphs for Boutroux curves}
\label{sec:StokesGraphsForBoutrouxCurves}
Let $\mathcal{R}$ be a Boutroux spectral curve of class $\{1111\}$, so $\mathcal{R}$ has genus one.  From the formula $Q(z):=h'(z)^2=\tfrac{1}{16}z^{-2}P(z)=\tfrac{1}{16}z^{-2}(z-\alpha)(z-\beta)(z-\gamma)(z-\delta)$ we see that $h'(z)$ may be considered to be a meromorphic function on $\mathcal{R}$ with purely real residues at the poles over $z=0,\infty$.  Therefore the only possible monodromy of $\mathrm{Re}(h(z))$ on $\mathcal{R}$ arises from the nontrivial homology of $\mathcal{R}$ as a curve of genus one.  But since $\mathcal{R}$ is Boutroux, there is no such monodromy and therefore $\mathrm{Re}(h(z))$ is determined up to an integration constant as a single-valued non-constant function on $\mathcal{R}$ that is harmonic away from the poles of $h'(z)$.  As in the beginning of Section~\ref{sec:h-analysis}, it follows that the quadratic differential $Q(z)\,\dd z^2$ has no divergent critical v-trajectories (i.e. v-trajectories having a zero of $Q(z)$ as an endpoint).   In the current setting we refer to the closure $\critclosure$ of the union of the critical v-trajectories of $Q(z)\,\dd z^2$ as the \emph{Stokes graph} of $\mathcal{R}$.  

The key feature contributed by the condition that $\mathcal{R}$ is Boutroux is that the Stokes graph of $\mathcal{R}$ coincides with the projection from $\mathcal{R}$ to the $z$-plane of the level curve $\mathrm{Re}(h(z))=\mathrm{Re}(h(\alpha))$ (or $=\mathrm{Re}(h(\beta))=\mathrm{Re}(h(\gamma))=\mathrm{Re}(h(\delta))$ by the Boutroux conditions \eqref{eq:Boutroux}).  To see this, we may follow the line of argument in Section~\ref{sec:K=LevelSet-t-plane}.  Clearly every point on the Stokes graph of $\mathcal{R}$ is on the level curve.  The complement of the Stokes graph is then the disjoint union of a single circle domain $\circledomain$ containing the double pole $z=0$ and finitely many end domains (strip domains and ring domains being excluded because they require any branch of $\mathrm{Re}(h(z))$ to take different values on different parts of their boundaries).  But each end domain is mapped conformally by a branch of $h(z)$ onto an open right or left half-plane while the circle domain is mapped conformally by a branch of $\ee^{h(z)}$ onto the interior or exterior of a circle.  Therefore in the interior of each end domain or circle domain, $\mathrm{Re}(h(z))$ is unequal to its constant value on the boundary.  So there are no points of the level set of $\mathrm{Re}(h(z))=\mathrm{Re}(h(\alpha))$ not contained in the Stokes graph of $\mathcal{R}$.

Since there are no strip domains and $Q(z)\,\dd z^2$ has a pole of order $6$ at $z=\infty$, a neighborhood of the point at infinity is covered by the disjoint union of four end domains and four critical v-trajectories separating them and approaching $z=\infty$ asymptotically in the diagonal directions $\arg(z)=\pm\tfrac{1}{4}\pi,\pm\tfrac{3}{4}\pi$.  These four are the only end domains since there are no other poles of $Q(z)\,\dd z^2$ of order greater than two, and the four end domains are mutually disjoint on the $z$-sphere (every end domain has exactly one such pole at a unique point on its boundary on the $z$-sphere).  On the $z$-sphere the Stokes graph is therefore a planar graph bounding exactly five faces (four end domains and one circle domain).  The vertices of the graph are the points $z=\alpha,\beta,\gamma,\delta,\infty$ of degrees $d(\alpha)=d(\beta)=d(\gamma)=d(\delta)=3$ and $d(\infty)=4$.  Based on the count of vertices and their degrees and the fact that $z=\infty$ is necessarily connected to a finite vertex, by enumeration it is easy to check that the Stokes graph of $\mathcal{R}$ is connected, and therefore by Euler's formula there are exactly 8 edges.  

Given the Stokes graph $\critclosure$ of $\mathcal{R}$, we may assume that $h'(z)$ is an analytic function in $\mathbb{C}\setminus\critclosure$ by suitably arranging the branch cut locus $B$ of $R(z)$ in the formula \eqref{eq:h-g-phi} to lie within $\critclosure$.  It then follows by integration along contours in $\mathbb{C}\setminus (B\cup\{0\})$ with fixed base point in the same domain that $\mathrm{Re}(h(z))$ becomes a well-defined single-valued function of $z$ that is harmonic on $\mathbb{C}\setminus (B\cup\{0\})$ and that extends continuously to $B$.  By choice of base point (integration constant) we may assume that $\mathrm{Re}(h(z))=0$ whenever $z\in \critclosure$, and then the sign of $\mathrm{Re}(h(z))$ will be well-defined in each component of $\mathbb{C}\setminus\critclosure$.  This essential property of $h(z)$ will play a key role in the steepest descent analysis in Section~\ref{sec:G1} below.

If we retain only the essential topological information, we associate to the Stokes graph of $\mathcal{R}$ an \emph{abstract Stokes graph}.  This is a kind of connected planar graph (allowing for loops) with four vertices of degree $3$ representing $z=\alpha,\beta,\gamma,\delta$ and four special vertices of degree $1$ representing the four diagonal directions at $z=\infty$.  We place the four latter vertices at the four corners of a bounding square (the edges of which are not considered as part of the graph) and require that all other edges lie within this square.  Each abstract Stokes graph has exactly $5$ faces on the square representing the four end domains and one circle domain, and by Euler's formula, there are exactly $8$ edges.  We identify two abstract Stokes graphs that are related by a homeomorphism of the bounding square and its interior that fixes the square.  Now since $\mathcal{R}$ depends on $y_0$ within a given Boutroux domain $\mathcal{B}$, the Stokes graph of $\mathcal{R}$ does as well.  However the abstract Stokes graph is an invariant, depending only on the choice of Boutroux domain $\mathcal{B}$.  This also implies that the branch points $\alpha$, $\beta$, $\gamma$, and $\delta$ may be consistently labeled on $\mathcal{B}$ by labeling them for just one value of $y_0\in\mathcal{B}$. 

\subsection{Degenerate Boutroux curves and their abstract Stokes graphs}
\label{sec:DegenerateStokesGraphs}
Let $\kappa\in (-1,1)$
%$\kappa\in (-1,1)\cup (1,+\infty)$ 
be fixed, and let $\mathcal{B}$ be a Boutroux domain.  The points on the boundary $\partial\mathcal{B}$ correspond to singular spectral curves of class $\{31\}$ or $\{211\}$.  
We call these \emph{degenerate Boutroux curves}.  Letting $y_0\in\mathcal{B}$ approach the boundary, the generic (i.e., corresponding to case $\{211\}$) mechanism of degeneration is that a pair of simple roots of $P$ coalesces into a double root $\gamma$ and there remain two distinct simple roots $\alpha$ and $\beta$.  If $\mathfrak{a}$ is the cycle that encloses the roots that coalesce into $\gamma$, then it is obvious that $\oint_\mathfrak{a}v\,\dd z\to 0$ in the limit, so the only nontrivial Boutroux condition is the limiting form of $\mathrm{Re}(\oint_\mathfrak{b}v\,\dd z)=0$, which can be written as $\mathrm{Re}(\Phi(t))=0$ (cf.\@ \eqref{eq:Phi-t}) where $t$ parametrizes pairs of points $\pm(y=y_0,\gamma)$ on the Riemann surface of \eqref{eq:gamma-eqn}.

Therefore, the image in the $y=y_0$-plane of the closure of the union of critical v-trajectories of $\Phi'(t)^2\,\dd t^2$, studied in detail in Section~\ref{sec:OutsideDomain}, consists of the only possible points where a degenerate Boutroux curve may exist.  Since the images in the $y_0$-plane of the points $t=a(w),b(w),\tau(w),\tau(w)^*$ are precisely the branch points solving $B(y_0;\kappa)=0$, these are the only points that admit a degenerate Boutroux curve of class $\{31\}$.  It follows that the images of the open critical v-trajectories themselves are the possible points admitting a degenerate Boutroux curve of class $\{211\}$.  

The Stokes graph of a degenerate Boutroux curve $\mathcal{R}$ is still defined as the closure $\critclosure$ of the union of critical v-trajectories of $Q(z)\,\dd z^2$ with $Q(z):=\tfrac{1}{16}z^{-2}P(z)$ but now the quartic $P(z)$ has fewer than four roots, being in case $\{31\}$ (two roots, one being simple) or $\{211\}$ (three roots, two being simple).  In the latter (more general) case, $\critclosure$ connects the double zero $z=\gamma$ of $Q(z)\,\dd z^2$ to either the simple zero $z=\alpha$ or the simple zero $z=\beta$ (or both).  The corresponding abstract Stokes graph has two vertices of degree $3$ (the simple zeros), one vertex of degree $4$ (the double zero), and the four special vertices of degree $1$ for the four directions at $z=\infty$.  It is planar, and has $5$ faces just like in the nondegenerate case.  By Euler's formula there are exactly $7$ edges.  

%By definition, the Stokes graph of a type (iv) degenerate Boutroux curve is the closure of the union of all critical vertical trajectories of $Q(z)\,\dd z^2$ where $Q(z)$ has two simple roots $z=\alpha,\beta$ and one double root $z=\gamma$, and where $z=\gamma$ is connected by a trajectory to either $z=\alpha$ or $z=\beta$ or both (this is the Boutroux condition $\mathrm{Re}(\Phi(t))=0$), and therefore the whole Stokes graph is connected with exactly four unbounded arcs tending to $z=\infty$ in the four diagonal directions.  Topologically, we may encode this information in an \emph{abstract Stokes graph}, which is a kind of connected planar graph (allowing for loops) with two vertices of degree $3$, one vertex of degree $4$ and four special vertices of degree $1$ representing the directions at infinity.  We place the four latter vertices at the four corners of a bounding square (the edges of which are not considered as part of the graph) and require that all other edges lie within this square.  Each abstract Stokes graph has exactly $5$ faces on the square representing four end domains and one circle domain.  By Euler's formula, there are exactly $7$ edges.  
There are ten such graphs, up to $\mathrm{Dih}_4$ symmetry of rotations and reflections that preserve the bounding square.  Representatives of each symmetry class are shown in Figure~\ref{fig:AbstractStokesGraphCatalog}.
\begin{figure}[h]
\begin{center}
\includegraphics{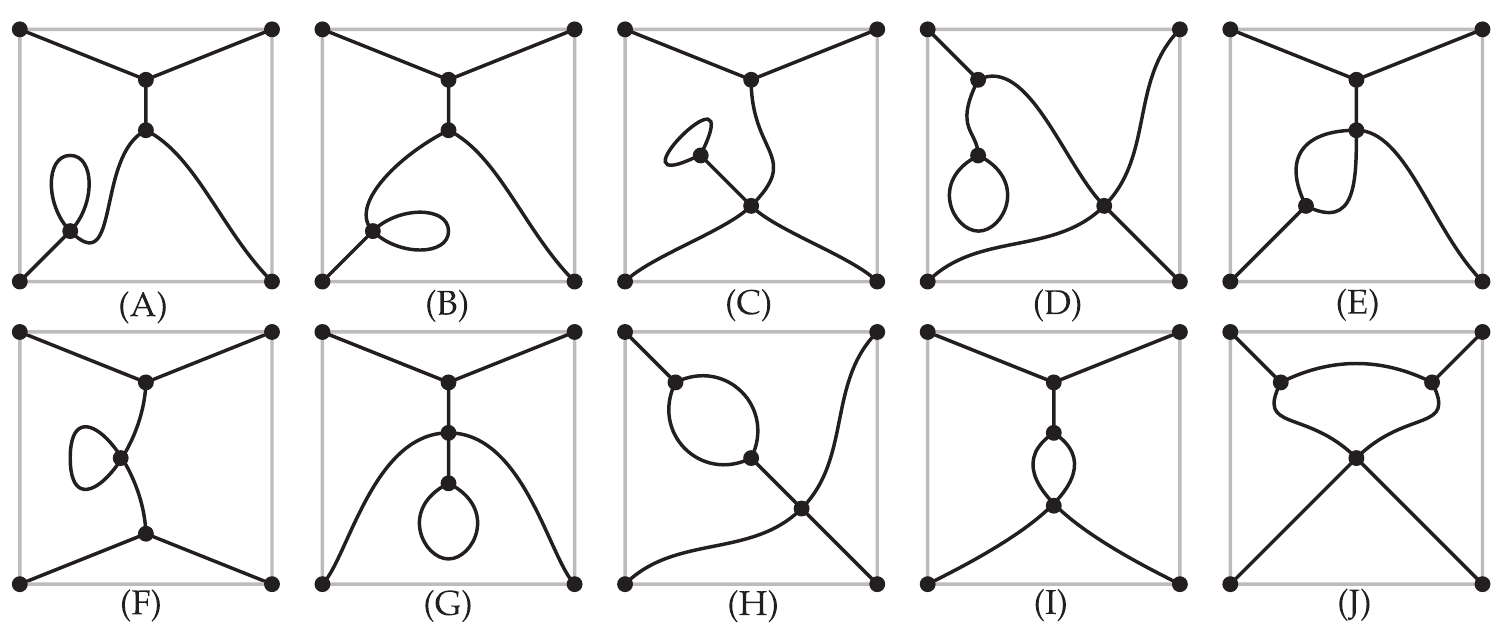}
\end{center}
\caption{The abstract Stokes graphs for degenerate Boutroux curves of class $\{211\}$, modulo $\mathrm{Dih}_4$ symmetry, drawn on their bounding squares (gray).  The graphs (A)--(E) in the first row have no symmetry (so $\mathrm{Dih}_4$ acts freely) and hence each generates $8$ abstract Stokes graphs each.  The graphs (F)--(J) in the second row have a reflection symmetry and hence each only generates $4$ abstract Stokes graphs by the $\mathrm{Dih}_4$ action.  The total number of abstract Stokes graphs is therefore $60$. 
The graphs (D), (F), and (I) can be ruled out for $t$ on critical v-trajectories of $\Phi'(t)^2\,\dd t^2$.} 
\label{fig:AbstractStokesGraphCatalog}
\end{figure}

Each critical v-trajectory of $\Phi'(t)^2\,\dd t^2$ carries a pair of abstract Stokes graphs for degenerate Boutroux curves of class $\{211\}$.  These two abstract Stokes graphs are related by reflection through the origin (this is the symmetry relating $(y_0,\gamma)$ and $(-y_0,-\gamma)$ both of which correspond to the same value of $t$).  We next determine which of the $60$ possible abstract Stokes graphs correspond to each critical v-trajectory of $\Phi'(t)^2\,\dd t^2$.  For this purpose, 
we divide the critical v-trajectories in the $t$-plane into groups:  the critical v-trajectories 
within 
%in the interior of 
each of the four domains $S_0=S_\mathrm{gO}$, $S_\infty=S_\mathrm{gH}$, $S_+$ and $S_-$ (two v-trajectories in each domain), the four v-trajectories on the boundary $\partial S_\mathrm{gO}$, and the two v-trajectories on the boundary $\partial S_\mathrm{gH}$.

%The results differ depending upon whether $\kappa\in (-1,1)$ or $\kappa>1$, and we treat these cases in turn.

%\subsubsection{Abstract Stokes graphs for critical v-trajectories of $\Phi'(t)^2\,\dd t^2$ when $\kappa\in (-1,1)$}
%\subsubsection*{Abstract Stokes graphs for critical v-trajectories in the interior of $S_0=S_\mathrm{gO}$}
\subsubsection{Abstract Stokes graphs for critical v-trajectories 
%in the interior of 
within 
$S_0=S_\mathrm{gO}$}
It is sufficient to consider $t$ small since the abstract Stokes graphs are the same for all points $t$ on a given v-trajectory.  When $t$ is small, $Q(z)$ has one small simple root, say $z=\alpha$, while the other simple root $z=\beta$ and the double root $z=\gamma$ are both large.  It can therefore be shown that the circle domain containing $z=0$ has only the small simple root $\alpha$ on its boundary, which implies that the face of the abstract Stokes graph that does not abut any edge of the bounding square is bounded by a loop edge connected to a vertex of degree $3$.  Moreover, since $t$ is real on the v-trajectories of interest, the Stokes graph in the $z$-plane has to be Schwarz symmetric either through the real or imaginary axes.  Hence the abstract Stokes graphs for the two v-trajectories 
within
%interior to 
$S_0=S_\mathrm{gO}$ are of the symmetry type of graph (G) in Figure~\ref{fig:AbstractStokesGraphCatalog}.
%, which does not give rise to any of the abstract Stokes graphs shown in Figure~\ref{fig:AbstractStokesGraphsRectangle-kappa0} via $\mathrm{Dih}_4$ symmetry.

%\subsubsection*{Abstract Stokes graphs for critical v-trajectories in the interior of $S_\infty=S_\mathrm{gH}$}
\subsubsection{Abstract Stokes graphs for critical v-trajectories 
within
%in the interior of 
$S_\infty=S_\mathrm{gH}$}
In this case it is sufficient to consider large $t$.  In this situation, both simple roots $z=\alpha,\beta$ of $Q(z)$ are small of the same order, while the double root $z=\gamma$ is large, so the face of the abstract Stokes graph that is separated from the bounding square is bounded by a cycle consisting of two edges joining the two small simple roots.  Hence the abstract Stokes graphs for the two v-trajectories 
within
%interior to 
$S_\infty=S_\mathrm{gH}$ are of the symmetry type of graph (H) in Figure~\ref{fig:AbstractStokesGraphCatalog}.
%, which is again inconsistent with the abstract Stokes graphs shown in Figure~\ref{fig:AbstractStokesGraphsRectangle-kappa0}.

%\subsubsection*{Abstract Stokes graphs for critical v-trajectories in the interior of $S_\pm$}
\subsubsection{Abstract Stokes graphs for critical v-trajectories 
within
%in the interior of 
$S_\pm$}
Now it is sufficient to consider $t$ near $t_\infty^\pm$.  In both limits we find that the double root $z=\gamma$ of $Q(z)$ is small while both simple roots $z=\alpha,\beta$ of $Q(z)$ are large, although their difference is comparable in magnitude to $1$.  It can be shown that the double root $z=\gamma$ is the only critical point on the boundary of the circle domain containing the origin in the $z$-plane, and that the two simple roots $z=\alpha,\beta$ are connected by a critical v-trajectory in the $z$-plane.  Hence the abstract Stokes graph has one edge connecting the two vertices of degree $3$, both of the remaining edges from one of those vertices and one of the remaining edges from the other join to three of the points at infinity, and the vertex of degree $4$ has a loop forming the face corresponding to the circle domain and two other edges going to the fourth point at infinity and to the second vertex of degree $3$.  Therefore, the abstract Stokes graphs for v-trajectories 
within
%in the interior of 
$S_\pm$ are of the same symmetry type 
as either graph (A) or (B) in Figure~\ref{fig:AbstractStokesGraphCatalog}.
%, neither of which generates under $\mathrm{Dih}_4$ action any of the abstract Stokes graphs shown in Figure~\ref{fig:AbstractStokesGraphsRectangle-kappa0}.

%\subsubsection*{Abstract Stokes graphs for the v-trajectories of $\partial S_\mathrm{gO}$}
\subsubsection{Abstract Stokes graphs for the v-trajectories of $\partial S_\mathrm{gO}$}
Next consider the four critical v-trajectories forming the boundary of the domain $S_0=S_\mathrm{gO}$.  To determine the possible abstract Stokes graphs associated with these, we start with a configuration of critical v-trajectories for the quadratic differential $Q(z)\,\dd z^2$ in the $z$-plane for a point $t$ in the interior of $S_\mathrm{gO}$ not on the real axis, as shown in Figure~\ref{fig:GenusZero-y-Diag-Domains}.  Then we consider the different ways that the double root  $z=\gamma$ can merge with the zero level set shown with black curves in that figure.  There are a total of six scenarios:  $\gamma$ can merge with either of the two accessible v-trajectories emanating from the simple root $z=\alpha$ and simultaneously merge with the loop $\partial\circledomain$ or with either side of the unbounded v-trajectory emanating from the other simple root $z=\beta$.   Two of these give abstract Stokes graphs in the symmetry class of graph (G) in Figure~\ref{fig:AbstractStokesGraphCatalog} and hence correspond to the point $t$ crossing the real axis instead of meeting the boundary of $S_\mathrm{gO}$.  The four remaining scenarios correspond to the boundary of $S_\mathrm{gO}$ and they give abstract Stokes graphs in the symmetry class of graphs (C) or (E) in Figure~\ref{fig:AbstractStokesGraphCatalog}.  
%These are not consistent under $\mathrm{Dih}_4$ action with any of the abstract Stokes graphs shown in Figure~\ref{fig:AbstractStokesGraphsRectangle-kappa0}.

%\subsubsection*{Abstract Stokes graphs for the v-trajectories of $\partial S_\mathrm{gH}$}
\subsubsection{Abstract Stokes graphs for the v-trajectories of $\partial S_\mathrm{gH}$}
Finally we consider the critical v-trajectories in the $t$-plane that form the boundary of $S_\infty=S_\mathrm{gH}$.  To determine the possible abstract Stokes graphs for these it is sufficient to take $t$ to to be real.  If $t>0$ is large, the two opposite values of the double root $z=\gamma$ of $Q(z)$ are large and real.  The simple roots $z=\alpha,\beta$ form a complex-conjugate pair of small magnitude connected by two critical v-trajectories that bound the circle domain containing $z=0$.  Finally, the v-trajectory structure in the $z$-plane is Schwarz symmetric with respect to the real axis.  It follows that when $t>0$ decreases to the value on $\partial S_\mathrm{gH}$, all three roots end up on the boundary of the circle domain, and we deduce that the abstract Stokes graph for the v-trajectory of $\partial S_\mathrm{gH}$ passing through the positive real $t$-axis is a rotation by $\pm\tfrac{1}{2}\pi$ of graph (J) in Figure~\ref{fig:AbstractStokesGraphCatalog}.  By \eqref{eq:ysquared-kappa-0} for $y=y_0$ this arc of $\partial S_\mathrm{gH}$ is mapped to the boundary arcs of $\HermiteExterior$ that cross the real axis in the $y_0$-plane.
%corresponds to either the first or the third panel in Figure~\ref{fig:AbstractStokesGraphsRectangle-kappa0}.  
If $t<0$ is large, the two opposite values of the double root $z=\gamma$ of $Q(z)$ are large and pure imaginary.  The simple roots $z=\alpha,\beta$ form a small anti-conjugate pair connected by v-trajectories forming the boundary of the circle domain containing $z=0$, and the v-trajectories in the $z$-plane are now Schwarz symmetric with respect to the imaginary axis.  Letting $t<0$ increase to the boundary of $S_\mathrm{gH}$ we see that the abstract Stokes graphs for the v-trajectory of $\partial S_\mathrm{gH}$ that crosses the negative real $t$-axis match graph (J) in Figure~\ref{fig:AbstractStokesGraphCatalog} or its reflection through the origin.  By \eqref{eq:ysquared-kappa-0} for $y=y_0$, this arc of $\partial S_\mathrm{gH}$ is mapped to the boundary arcs of $\HermiteExterior$ that cross the imaginary axis in the $y_0$-plane.

\subsection{The Boutroux domain $\rectangle(\kappa)$}
\label{sec:rectangle-domain}
We now begin to describe the various Boutroux domains in the $y_0$-plane, starting with a domain that contains the origin.
\subsubsection{Boutroux curves for $y_0=0$}  
\label{sec:Boutroux-curves-y0-zero}
We claim that if $y_0=0$, taking $E=0$ makes $\mathcal{R}$ a Boutroux curve 
%for all $\kappa\in (-1,1)$ or $\kappa>1$ .  
for all $\kappa\in (-1,1)$.
Indeed, if $y_0=E=0$, then $P(z)=z^4+8\kappa z^2+16$,
the roots of which 
%.  
%The configuration of roots of $P(z)$ (branch points) then depends on whether $-1<\kappa<1$ or $\kappa>1$ as follows.
%\begin{itemize}
%\item
%If $-1<\kappa<1$, the roots of $P(z)$ for $y_0=E=0$ 
satisfy $z^2=-4\kappa\pm 4\ii\sqrt{1-\kappa^2}$, forming a complex-conjugate pair on the circle of radius $|z^2|=4$. 
%Since the values of $z^2$ form a complex-conjugate pair on the circle of radius $|z^2|=4$, 
Therefore,
the four roots of $P(z)$ form a quartet in the complex $z$-plane, symmetric under reflection in both the real and imaginary axes, and lying on the circle of radius $|z|=2$.  
Recalling the parametrization $\kappa=\sin(\tfrac{1}{2}\varphi)$, $\sqrt{1-\kappa^2}=\cos(\tfrac{1}{2}\varphi)$, $-\pi<\varphi<\pi$,
%If we parametrize $\kappa\in (-1,1)$ by $\kappa=-\cos(2\mu)$ for $0<\mu<\tfrac{1}{2}\pi$, 
the 
%branch points 
roots of $P(z)$ 
are exactly $z=\alpha,\beta,\gamma,\delta$, where 
\eq
\alpha=2\ee^{\ii\mu},\quad\beta=-2\ee^{-\ii\mu},\quad \gamma=-2\ee^{\ii\mu},\quad
\text{and}\quad\delta=2\ee^{-\ii\mu},\quad\mu:=\tfrac{1}{4}(\varphi+\pi)\in (0,\tfrac{1}{2}\pi).
\endeq
These are arranged in counterclockwise order about the origin.
We consider concrete $\mathfrak{a}$ and $\mathfrak{b}$ cycles satisfying the canonical intersection condition $\mathfrak{a}\circ\mathfrak{b}=1$ such as illustrated in the left-hand panel of Figure~\ref{fig:a-b-cycles}.  
%\item
%If $\kappa>1$, the roots of $P(z)$ for $y_0=E=0$ satisfy $z^2=-4\kappa\pm 4\sqrt{\kappa^2-1}<0$.  Therefore, there are four distinct purely imaginary branch points.  If we parametrize $\kappa>1$ by $\kappa=\cosh(2\mu)$ for $\mu>0$, then the branch points are exactly $z=\alpha,\beta,\gamma,\delta$, where
%\eq
%\alpha=2\ii\ee^{\mu},\quad
%\beta=2\ii\ee^{-\mu},\quad
%\gamma=-2\ii\ee^{-\mu},\quad\text{and}\quad
%\delta=-2\ii\ee^{\mu}.
%\endeq
%These are arranged in order down the imaginary axis.  We consider concrete $\mathfrak{a}$ and $\mathfrak{b}$ cycles satisfying the canonical intersection condition $\mathfrak{a}\circ\mathfrak{b}=1$ such as illustrated in the right-hand panel of Figure~\ref{fig:a-b-cycles}.  
%\end{itemize}
\begin{figure}[h]
\begin{center}
\includegraphics{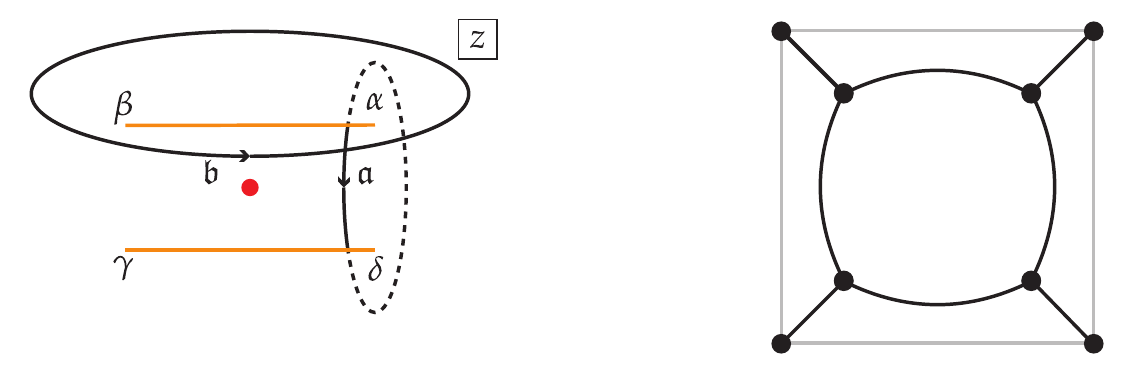}
\end{center}
\caption{Left panel:  the $\mathfrak{a}$ and $\mathfrak{b}$ cycles (solid black curves) on the sheet of $\mathcal{R}$ on which $v$ is a single-valued analytic function of $z$ with branch cuts (orange) connecting the branch points in pairs, and such that $v\sim \tfrac{1}{4}z$ as $z\to\infty$.  Part of the $\mathfrak{a}$ cycle (dashed) is on the other sheet, on which $v$ has the opposite sign.  The red dot is the origin.  
%Left:  $-1<\kappa<1$.  Right:  $\kappa>1$.
Right panel:  the abstract Stokes graph of the Boutroux domain $\rectangle$.
}
\label{fig:a-b-cycles}
\end{figure}
Clearly, $P(z)$ is an even function of $z$
% in both cases, and in both cases 
and we may define 
$v(z)$ as a single-valued analytic function with horizontal branch cuts between $\alpha$ and $\beta$ and between $\gamma$ and $\delta$, and normalized so that $v(z)=\tfrac{1}{4}z+\bo(z^{-1})$ as $z\to\infty$, and then we take the integrand to be $v(z)$ (resp., $-v(z)$) on the solid arcs (resp., dashed arcs) in 
the left-hand panel of Figure~\ref{fig:a-b-cycles}.  

Now we compute the integrals in the Boutroux conditions \eqref{eq:Boutroux}.
To compute the integral over the cycle $\mathfrak{a}$ we use the fact that $v(z^*)=v(z)^*$.  
%In the case $-1<\kappa<1$ we 
We
can choose $\mathfrak{a}$ to be Schwarz-symmetric up to orientation, and then it is obvious that $\oint_\mathfrak{a}v\,\dd z$ is purely imaginary.  
%In the case $\kappa>1$ we observe that $v(z)$ is purely real on both edges of both branch cuts.  Therefore, deforming $\mathfrak{b}$ by Cauchy's theorem to coincide with the oppositely-oriented opposite sides of the branch cut connecting $\alpha$ and $\beta$, we easily see that $\oint_\mathfrak{b}v\,\dd z$ is purely imaginary.  
To compute the integral over the cycle $\mathfrak{b}$, we use the fact that $v(-z)=-v(z)$.  
%In the case $-1<\kappa<1$ we 
We
observe that up to purely imaginary residue contributions we can identify $\oint_\mathfrak{b}v\,\dd z$ with $-\oint_{-\mathfrak{b}}v\,\dd z=-\oint_\mathfrak{b} v\,\dd z$ by Cauchy's Theorem, and we conclude that $\oint_\mathfrak{b}v\,\dd z$ is imaginary.  
%In the case $\kappa>1$, we may express $\oint_\mathfrak{a}v\,\dd z$ as twice the Cauchy principal value integral of $v$ along the straight line path from $\beta$ to $\gamma$ through the origin; by oddness of $v(z)$, this integral vanishes identically.  
Therefore, as both integrals are purely imaginary, $\mathcal{R}$ is indeed a Boutroux curve when $y_0=E=0$ 
%and $-1<\kappa<1$ or $\kappa>1$.
for all $\kappa\in (-1,1)$.

\subsubsection{The Boutroux domain $\rectangle$ and its abstract Stokes graph}
\label{sec:AbstractStokesGraphsForRectangle}
Associated with the continuation of the Boutroux curve obtained for $y_0=0$ and arbitrary $\kappa\in (-1,1)$ in Section~\ref{sec:Boutroux-curves-y0-zero} is a Boutroux domain denoted $\rectangle=\rectangle(\kappa)$ containing $y_0=0$.  It is easy to see by the substitution $z\mapsto \ii z$ in the integrand that the Boutroux equations \eqref{eq:Boutroux} are invariant under $(y_0,\kappa,E)\mapsto (\ii y_0,-\kappa,-\ii E)$; since the base point $(y_0,E)=(0,0)$ is fixed by this symmetry, it follows by uniqueness of continuation that $E(y_0;-\kappa)=\ii E(\ii y_0;\kappa)$ holds for the family of functions $y_0\mapsto E$ defined by continuation on $\rectangle(\kappa)$ for all $\kappa\in (-1,1)$.  Our notation suggests the shape of this domain $\rectangle(\kappa)$ in the $y_0$-plane whose boundary we shall see is a curvilinear rectangle.  

Before describing that boundary, we work out the abstract Stokes graph for the Boutroux domain $\rectangle$.
%, which are different depending on whether $\kappa\in (-1,1)$ or $\kappa>1$.  
We proceed by setting $y_0=0$ and deducing the Stokes graph in this special case.  This is sufficient to determine the abstract Stokes graph for all $y_0\in\rectangle$ as it is independent of $y_0$.  Recall that the Stokes graph $\critclosure$ is the closure of the union of the three critical v-trajectories emanating from each of the four zeros $z=\alpha,\beta,\gamma,\delta$.  Each of these v-trajectories must 
either terminate at one of the four zeros (possibly returning to the same zero a priori, although we will shortly rule that out) or escape to $z=\infty$ in one of the four directions $\arg(z)=\pm\tfrac{1}{4}\pi,\pm\tfrac{3}{4}\pi$.  Each of the latter directions accepts exactly one critical v-trajectory.  Also recall that exactly one of the five maximal connected components of $\overline{\mathbb{C}}\setminus\critclosure$ is a circle domain $\circledomain$ that has at least one of the four zeros $z=\alpha,\beta,\gamma,\delta$ on its boundary $\partial\circledomain$ that is a Jordan curve.  
%We denote that boundary, a Jordan curve in $\mathbb{C}$, by $C=\partial\circledomain$.

Now observe that if $y_0=0$ and $E=0$ as in Section~\ref{sec:Boutroux-curves-y0-zero}, then the quadratic differential $Q(z)\,\dd z^2$ enjoys Schwarz reflection symmetry in both the real and imaginary $z$-axes.  %In the case $-1<\kappa<1$ we immediately see that, since
Since
the group of reflection symmetries acts freely on the four roots of $P(z)$ in this case, all four branch points lie on the boundary $\partial\circledomain$ of the circle domain $\circledomain$.  Thus, there are four critical v-trajectories whose closure is $\partial\circledomain$, and that connect the four zeros in pairs.  It remains to determine the fate of the remaining critical v-trajectory emanating from each of the four zeros.  By Lemma~\ref{lem:Teichmueller} none of these can coincide with any of the others (necessarily forming, with part of $\partial\circledomain$, a Jordan curve of v-trajectories and their endpoints with only regular points in its interior so that referring to \eqref{eq:Teichmueller-L}--\eqref{eq:Teichmueller-R} $L\le 0$ while $R=2$), so all four of them must go to $z=\infty$.  By symmetry, the v-trajectory emanating from each zero tends to infinity in the same quadrant that contains the zero.   This completes the qualitative description of the Stokes graph for $y_0=0$ with $E=0$ for all $\kappa\in (-1,1)$.  Numerical plots confirm this qualitative description; see the upper left-hand panel of Figure~\ref{fig:SignsContourLenses-y0-kappa0-s1} or Figure~\ref{fig:SignsContourLenses-y0-kappa0-sm1} below for the Stokes graph for $y_0=0$ and $\kappa=0$.  The corresponding abstract Stokes graph for the whole Boutroux domain $\rectangle$ is shown in the 
%left-hand panel of Figure~\ref{fig:ASGs-rectangle}.
right-hand panel of Figure~\ref{fig:a-b-cycles}.

\subsubsection{The boundary of $\rectangle$}
Recall that the boundary of a Boutroux domain $\mathcal{B}$ consists of arcs that are the images in the $y=y_0$-plane of critical v-trajectories of $\Phi'(t)^2\,\dd t^2$ described in Section~\ref{sec:OutsideDomain}.  However, not all of these curves will be relevant to describe the boundary of the specific Boutroux domain $\rectangle$ containing $y_0=0$.  The reason that not every image of a critical v-trajectory is relevant is that points on the boundary of $\rectangle$ can only correspond to class $\{211\}$ degenerate Boutroux curves having Stokes graphs in the $z$-plane of certain specific symmetry types among the many types associated with all of the critical v-trajectories in the $t$-plane as described in Section~\ref{sec:DegenerateStokesGraphs}.  

Looking at 
%Figure~\ref{fig:ASGs-rectangle} 
the right-hand panel of Figure~\ref{fig:a-b-cycles}
and considering degenerations in which pairs of finite vertices fuse such that one edge disappears while the number of faces is preserved, we can deduce that exactly four abstract Stokes graphs can correspond to points on the boundary of $\rectangle$ other than the eight roots of $B(y_0;\kappa)=0$.
%, but the graphs are different for $-1<\kappa<1$ (see Figure~\ref{fig:AbstractStokesGraphsRectangle-kappa0}) than for $\kappa>1$ (see Figure~\ref{fig:AbstractStokesGraphsRectangle-kappa3}).  
See Figure~\ref{fig:AbstractStokesGraphsRectangle-kappa0}; numerically-generated examples of actual degenerate Stokes graphs for points on the boundary of $\rectangle(\kappa)$ can be seen in Figures~\ref{fig:Trajectories-gO-k0}--\ref{fig:Trajectories-gH-k0}.
\begin{figure}[h]
\begin{center}
\includegraphics{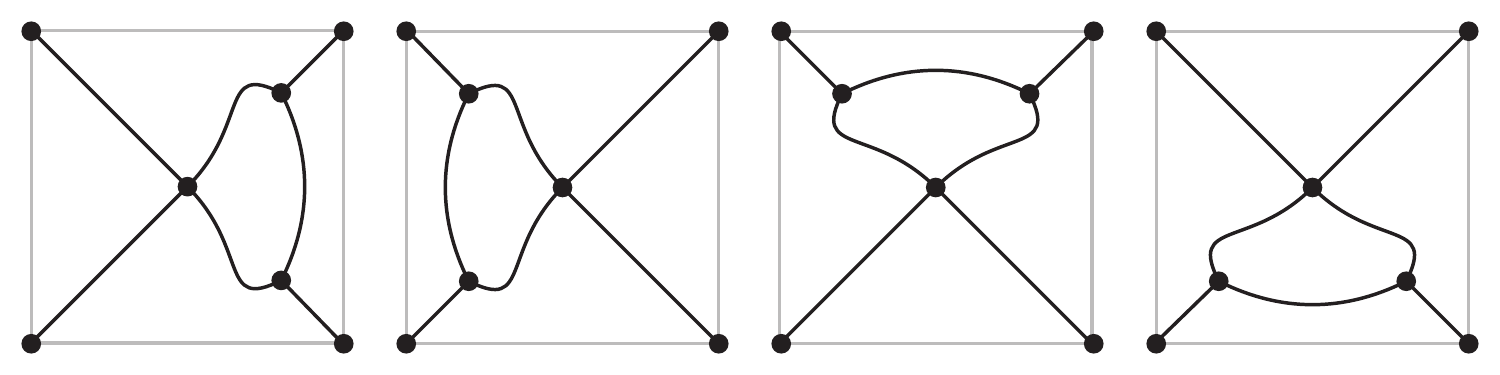}
\end{center}
%\caption{For $\kappa\in (-1,1)$, the 
\caption{The
four abstract Stokes graphs that can occur at boundary points of $\rectangle$ in the $y_0$-plane other than the branch points $B(y_0;\kappa)=0$.  These are all of the symmetry type of graph (J) in Figure~\ref{fig:AbstractStokesGraphCatalog}.}
\label{fig:AbstractStokesGraphsRectangle-kappa0}
\end{figure}
%\begin{figure}[h]
%\begin{center}
%\includegraphics{AbstractStokesGraphsRectangle-kappa3.pdf}
%\end{center}
%\caption{As in Figure~\ref{fig:AbstractStokesGraphsRectangle-kappa0} but for $\kappa>1$.  These are of the symmetry types of graphs (F) and (I) in Figure~\ref{fig:AbstractStokesGraphCatalog}.}
%\label{fig:AbstractStokesGraphsRectangle-kappa3}
%\end{figure}
Since only graphs of symmetry type (J) appear,
% for $\kappa\in (-1,1)$ while only graphs of symmetry types (F) and (I) appear for $\kappa>1$, from the analysis in Section~\ref{sec:DegenerateStokesGraphs} we see that
%in both cases, 
the analysis in Section~\ref{sec:DegenerateStokesGraphs} shows that
the four $y_0$-plane images of the two critical v-trajectories of $\Phi'(t)^2\,\dd t^2$ on the boundary of $S_\infty=S_\mathrm{gH}$ form, with their endpoints, the boundary of the Boutroux domain $\rectangle$ containing $y_0=0$.  The closure of the union of these four image arcs is a Jordan curve because $y^2$ is univalent on $S_\mathrm{gH}$, and it is symmetric with respect to $y_0\mapsto y_0^*$ and $y_0\mapsto -y_0$.  $\rectangle$ is the interior of this Jordan curve, and its boundary is a curvilinear rectangle, as claimed earlier.  Moreover, recalling Definition~\ref{def:ExteriorDomains}, $\rectangle=\mathbb{C}\setminus\overline{\HermiteExterior(\kappa)}$.  See the upper right-hand panel of Figure~\ref{fig:Domains}.
%, where we recall that $\HermiteExterior$ is the unbounded domain on which an analogue of the asymptotic formula \eqref{eq:exterior-asymptotic} holds for the rational solutions of generalized Hermite type.  

Note that $\rectangle$ is also disjoint from $\OkamotoExterior(\kappa)$, the 
%generalized Okamoto 
other ``exterior'' domain defined in Definition~\ref{def:ExteriorDomains}.  This is true because in Section~\ref{sec:G1} below we will prove the asymptotic formula \eqref{eq:Okamoto-elliptic} in Theorem~\ref{thm:Okamoto-elliptic} that describes the gO Painlev\'e-IV rational solutions for $y_0\in\rectangle$ and that is inconsistent with Theorem~\ref{thm:OkamotoExterior} if also $y=y_0\in\OkamotoExterior(\kappa)$.
%
%Indeed, the latter is precisely the domain on which the generalized Okamoto Painlev\'e-IV rationals have asymptotic behavior given by \eqref{eq:exterior-asymptotic}.  If any point in this domain were also to lie in $\rectangle$, then the alternate asymptotic formula \textcolor{red}{(give a more precise reference later when the rectangle genus $1$ stuff is sufficiently well-developed)} in terms of elliptic functions is also valid, which is a contradiction.  
In other words, $\rectangle\subset \mathbb{C}\setminus\overline{\OkamotoExterior}(\kappa)$.  Moreover, an argument like that mentioned 
in Section~\ref{sec:ExteriorOkamotoUnivalence}
%in the footnote on page~\pageref{page:Jordan} 
shows that 
the only points common to $\partial\rectangle$ and the Jordan curve $\partial\OkamotoExterior(\kappa)$ passing through all eight solutions of $B(y_0;\kappa)=0$ are the four corner points of $\rectangle$, namely the images in the $y_0$-plane of $t=\tau,\tau^*$.  This makes the region in between $\partial\rectangle$ and $\partial\OkamotoExterior(\kappa)$ the disjoint union of four simply connected domains whose boundaries are curvilinear triangles.  Each of these triangles consists of three arcs that are images in the $y_0$-plane of the boundary arcs of $S_+$ 
%(the triangles having a vertex on the real axis for $\kappa\in (-1,1)$ and on the imaginary axis for $\kappa>1$) 
(the triangles having a vertex on the real axis) 
and $S_-$ 
%(vice-versa).
(the triangles having a vertex on the imaginary axis). 
We will next show that the interior of each of these four curvilinear triangles is a Boutroux domain, and we will find the abstract Stokes graphs for these domains. By the symmetries noted in Proposition~\ref{prop:FirstQuadrant}, it suffices to consider just two of the triangles, one each on the positive real and imaginary axes.

\subsection{The Boutroux domain $\TR(\kappa)$}
\label{sec:TR}
Let $\TR=\TR(\kappa)$ denote the maximal connected component of $\mathbb{C}\setminus\overline{\rectangle\cup\OkamotoExterior(\kappa)}$ that is entirely contained within the right half $y_0$-plane, a simply connected domain whose boundary is a curvilinear triangle with vertices agreeing with those of an equilateral triangle (cf.\@ Proposition~\ref{prop:triangles}), one vertex of which lies on the positive real axis. See the upper right-hand panel of Figure~\ref{fig:Domains}.  We claim that $\TR$ is a Boutroux domain.

To show this, we need to find a smooth mapping $\TR\ni y_0\mapsto (E_\mathrm{R},E_\mathrm{I})\in \mathbb{R}^2$ such that the quartic $P(z)$ determines a Boutroux curve $\mathcal{R}$ of class $\{1111\}$.  Unfortunately, the approach used in Section~\ref{sec:rectangle-domain} of guessing a solution of the Boutroux conditions \eqref{eq:Boutroux} for a convenient value of $y_0$ such as $y_0=0$ to perturb from using the Implicit Function Theorem is not helpful.  Instead, we devise a strategy based on single-variable calculus that allows us to obtain a Boutroux curve $\mathcal{R}$ for any $y_0$ confined to the real axis.   

\subsubsection{Boutroux curves for $y_0\in\mathbb{R}$}
\label{sec:R-to-TR}
%Given that $\kappa\in (-1,1)\cup (1,+\infty)$ is real, 
Given $\kappa\in (-1,1)$,
we will show that whenever $y_0\in\mathbb{R}$ there is a unique (possibly degenerate) Boutroux curve $\mathcal{R}$ for which $E$ is real, i.e., $E_\mathrm{I}=0$.  Indeed, 
%if $\kappa\in\mathbb{R}$, $y_0\in\mathbb{R}$, and $E\in\mathbb{R}$, 
if both $y_0$ and $E$ are real
it follows easily that the quartic polynomial $P(z)$ is Schwarz symmetric:  $P(z^*)=P(z)^*$.  Let us assume for the moment that the spectral curve for such $y_0$ and $E$ is of class $\{1111\}$ so that $P(z)$ has four distinct roots.  These roots necessarily come in complex-conjugate or real pairs.  

We first simplify the Boutroux conditions \eqref{eq:Boutroux} under the additional assumption that there are no real roots, and without loss of generality we let $\alpha$ and $\beta$ denote the roots in the upper half-plane, and assume that $\gamma=\beta^*$ and that $\delta=\alpha^*$.  We define $v(z)$ as a function satisfying $v^2=P(z)/(16z^2)$ that is analytic apart from $z=0$ and a Schwarz-symmetric system of two branch cuts, one of which lies in the upper half-plane joining $\alpha$ and $\beta$, and that satisfies $v(z)=\tfrac{1}{4}z+\tfrac{1}{2}y_0 +\kappa z^{-1}+ \bo(z^{-2})$ as $z\to\infty$.  Then we may select a canonical homology basis that topologically matches the diagram in the left-hand panel of Figure~\ref{fig:a-b-cycles}.  Since $v(z^*)=v(z)^*$, by exactly the same argument as in Section~\ref{sec:Boutroux-curves-y0-zero} we deduce that $f_\mathfrak{a}(E_\mathrm{R},0;y_0,\kappa)=0$ holds automatically.  However unlike the case considered in Section~\ref{sec:Boutroux-curves-y0-zero}, the additional symmetry $v(-z)=-v(z)$ is generally no longer present so the remaining Boutroux condition is not automatically satisfied.   To simplify this condition,
we first apply Cauchy's Theorem to write
\eq
\oint_\mathfrak{b}v(z)\,\dd z = \oint_\mathfrak{b}\left[v(z)-\frac{1}{4}z-\frac{1}{2}y_0\right]\,\dd z.
\label{eq:Cauchy-identity}
\endeq
The integrand satisfies $v(z)-\tfrac{1}{4}z-\tfrac{1}{2}y_0 = \kappa z^{-1}+\bo(z^{-2})$ as $z\to\infty$ and has a simple pole at $z=0$ with real residue, so by enlarging the loop $\mathfrak{b}$ to consist of a path from $z=-L$ to $z=L$ indented above the pole at the origin together with a semicircular path from $z=L$ to $z=-L$ in the upper half-plane, and letting $L\to +\infty$ we find that 
\eq
\begin{split}
f_\mathfrak{b}(E_\mathrm{R},0;y_0,\kappa)&=\mathrm{P.V.}\int_\mathbb{R}\left[v(z)-\frac{1}{4}z-\frac{1}{2}y_0\right]\,\dd z \\ &= \mathrm{P.V.}\int_\mathbb{R}\left[\frac{\sqrt{P(z)}}{4z}-\frac{1}{4}z-\frac{1}{2}y_0\right]\,\dd z
\end{split}
\label{eq:Boutroux-b-real}
\endeq
where the positive square root is meant, and the principal value exclusions occur both at $z=0$ and $z=\infty$.  

%The above statements presume that there are no real roots of $P(z)$.
%To continue this analysis to larger real values of $y$ and $E$, it is necessary to consider other configurations of roots of $P(z)$ for which either there exists one real pair and one complex-conjugate pair, or for which all four roots of $P(z)$ are real.  We claim that given $y\in\mathbb{R}$,  $E=E_\mathrm{R}\in\mathbb{R}$, and $c<-1$, $\Gamma$ is a Boutroux curve provided that $P(z)$ has distinct roots and that $f_\mathfrak{b}(E_\mathrm{R},0;y,c)=0$ in which the definition \eqref{eq:Boutroux-b-real} is modified only by multiplying $v(z)$ or $\sqrt{P(z)}$ in the integrand by the characteristic function of the intervals of $z$ on which $P(z)>0$.  

%To prove the claim, first suppose that $P(z)$ has exactly two real roots, in a configuration that arises from that shown in Figure~\ref{fig:a-b-cycles} by real homotopy in $y$ and $E_\mathrm{R}$.  

Next suppose that $P(z)$ has exactly two real roots.
Then it follows from the fact that $P(0)=16\neq 0$ that both real roots have the same sign.  We define a single-valued function $v(z)$ by $v(z)^2=P(z)/(16z^2)$, $v(z)=\tfrac{1}{4}z + \tfrac{1}{2}y + \bo(z^{-1})$ as $z\to\infty$ with $v(z)$ analytic except for a branch cut chosen to connect the two real roots of $P(z)$ and another connecting the conjugate pair of roots of $P(z)$, symmetric with respect to Schwarz reflection, and crossing the real axis vertically at a unique point between the origin and the real pair of roots.  Then we equip $\mathcal{R}$ with a canonical homology basis $(\mathfrak{a},\mathfrak{b})$ with representatives taken as in the left-hand panel of Figure~\ref{fig:general-real-configurations}.
\begin{figure}[h]
\begin{center}
\includegraphics{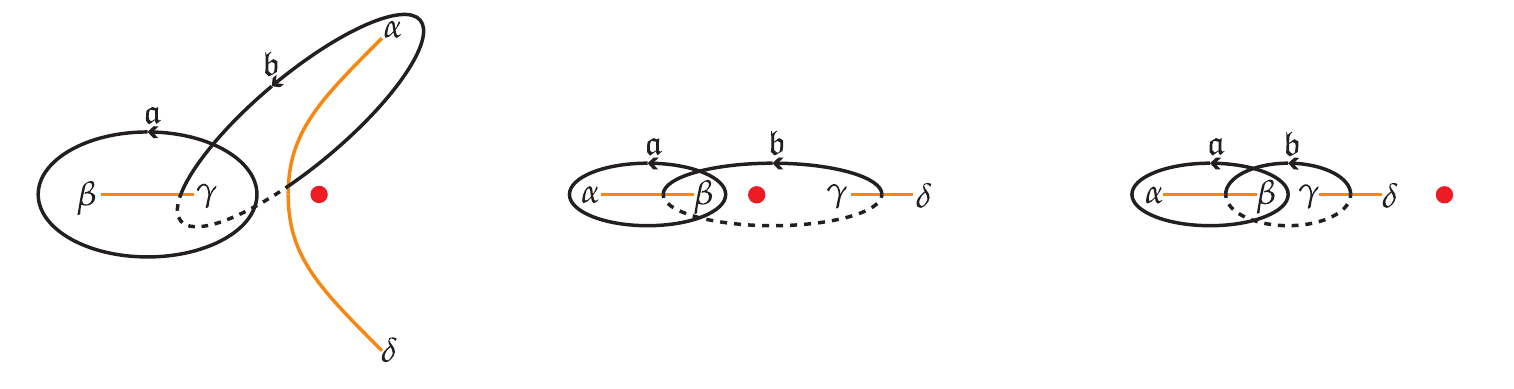}
\end{center}
\caption{As in the right-hand panel of Figure~\ref{fig:a-b-cycles} except for more general configurations of $\mathcal{R}$ occurring for $y_0\in\mathbb{R}$ and $E=E_\mathrm{R}\in\mathbb{R}$.  Left panel:  one real pair of roots.  Center panel:  two opposite-sign real pairs of roots.  Right panel:  two same-sign real pairs of roots.}
\label{fig:general-real-configurations}
\end{figure}
Since $v(z)=v(z^*)^*$ takes imaginary boundary values on the real branch cut $[\beta,\gamma]$, $\oint_\mathfrak{a}v\,\dd z$ is purely imaginary so once again the condition $f_\mathfrak{a}(E_\mathrm{R},0;y_0,\kappa)$ is satisfied automatically.  Again using the identity \eqref{eq:Cauchy-identity}, we deform the cycle $\mathfrak{b}$ partly toward $z=\infty$ in the upper half-plane and partly toward the real axis from the upper half-plane.  Using that the residues at $z=0$ and $z=\infty$ are purely real, upon taking the real part of $\oint_\mathfrak{b} v\,\dd z$ we deduce the expression on the second line of \eqref{eq:Boutroux-b-real} for $f_\mathfrak{b}(E_\mathrm{R},0;y_0,\kappa)$ in which the square root $\sqrt{P(z)}$ is multiplied by the characteristic function of $\mathbb{R}\setminus [\beta,\gamma]$, i.e., the real intervals over which $P(z)>0$ holds.

Finally, suppose that $P(z)$ has four real roots.
% and arises from a deformation of $y\in\mathbb{R}$ and $E=E_\mathrm{R}\in\mathbb{R}$ of the configuration shown in Figure~\ref{fig:a-b-cycles}.  
Then since $P(0)=16> 0$ there are either two roots of each sign, or all roots have the same sign; we take the branch cuts of $v(z)$ and homology basis representatives as shown in the center and right-hand panels of Figure~\ref{fig:general-real-configurations} respectively.  In both cases the fact that the boundary values of $v(z)=v(z^*)^*$ on the branch cut $[\alpha,\beta]$ are imaginary implies that the condition $f_\mathfrak{a}(E_\mathrm{R},0;y_0,\kappa)=0$ holds automatically.  And in both cases the use of \eqref{eq:Cauchy-identity} admits a deformation of $\mathfrak{b}$ toward $z=\infty$ in the upper half-plane and toward the real axis elsewhere.  Using reality of the residues at $z=0,\infty$ we arrive again at the expression in the second line of \eqref{eq:Boutroux-b-real} for $f_\mathfrak{b}(E_\mathrm{R},0;y_0,\kappa)$ in which the square root $\sqrt{P(z)}$ is multiplied by the characteristic function of the union of intervals $(-\infty,\alpha)\cup (\beta,\gamma)\cup (\delta,+\infty)$, i.e., where $P(z)>0$ holds on the real line.  

%This completes the proof of the claim.
We conclude that in all cases that $\mathcal{R}$ is a class $\{1111\}$ spectral curve corresponding to real $y_0$ and $E$, the Boutroux conditions \eqref{eq:Boutroux} are satisfied provided only that $E=E_\mathrm{R}\in\mathbb{R}$ satisfies the real equation
\eq
f_\mathfrak{b}(E_\mathrm{R},0;y_0,\kappa)=\mathrm{P.V.}\int_\mathbb{R}
\left[\chi_{P(z)>0}(z)\frac{\sqrt{P(z)}}{4z}-\frac{1}{4}z-\frac{1}{2}y_0\right]\,\dd z=0.
\label{eq:fb-real-chi}
\endeq
Considering various limits in which a class $\{1111\}$ curve degenerates to class $\{31\}$ or class $\{211\}$, we see that \eqref{eq:fb-real-chi} is also necessary for the degeneration to satisfy the Boutroux conditions when $P(z)=P(z^*)^*$ has fewer than four distinct roots.  

Given $\kappa\in(-1,1)$ and $y_0\in\mathbb{R}$, $f_\mathfrak{b}(E_\mathrm{R},0;y_0,\kappa)$ is a continuous function of $E_\mathrm{R}\in\mathbb{R}$, in which the dependence on $E=E_\mathrm{R}$ in the integrand enters via $P(z)$ in the form \eqref{eq:elliptic-ODE}.  Moreover, it follows from \eqref{eq:fb-real-chi} that $f_\mathfrak{b}(E_\mathrm{R},0;y_0,\kappa)$ is differentiable with respect to $E_\mathrm{R}$ whenever the roots of $P$ are distinct, and that for such $E_\mathrm{R}$,
\eq
\frac{\partial f_\mathfrak{b}}{\partial E_\mathrm{R}}(E_\mathrm{R},0;y_0,\kappa)=\frac{1}{4}\int_\mathbb{R}\chi_{P(z)>0}(z)\frac{\dd z}{\sqrt{P(z)}},
\label{eq:Boutroux-b-real-derivative}
\endeq
where the integral is absolutely convergent and positive.  A simple discriminant calculation shows that points $E=E_\mathrm{R}\in\mathbb{R}$ for which the polynomial $P(z)$ has fewer than four distinct roots are isolated.  It is then clear that as $E_\mathrm{R}$ approaches any value that makes $\mathcal{R}$ degenerate, the derivative \eqref{eq:Boutroux-b-real-derivative} diverges to $+\infty$.  Thus for each and $\kappa\in (-1,1)$ and $y_0\in\mathbb{R}$, the function $E_\mathrm{R}\mapsto f_\mathfrak{b}$ is continuous and strictly monotone increasing on $\mathbb{R}$.  
Moreover, we have the following asymptotic behavior:
\begin{equation}
f_\mathfrak{b}(E_\mathrm{R},0;y_0,\kappa)=\pm M|E_\mathrm{R}|^{2/3}+o(|E_\mathrm{R}|^{2/3}),\quad E_\mathrm{R}\to \pm\infty,
\end{equation}
for all fixed real $y_0$ and $\kappa$, where the constant $M$ is given by (taking the square roots to be positive)
\eq
M:=
\int_{-\infty}^{-2^{1/3}}\left(-\sqrt{\frac{w^4+2w}{16w^2}}-\frac{w}{4}\right)\,\dd w
+\int_{-2^{1/3}}^0\left(-\frac{w}{4}\right)\,\dd w + \int_0^{+\infty}\left(\sqrt{\frac{w^4+2w}{16w^2}}-\frac{w}{4}\right)\,\dd w\approx 1.25203>0.
\endeq
Hence, by strict monotonicity and the Intermediate Value Theorem, there exists exactly one root $E_\mathrm{R}=E_\mathrm{R}(y_0,\kappa)\in\mathbb{R}$ for which $f_\mathfrak{b}(E_\mathrm{R},0;y_0,\kappa)=0$.  
%See Figure~\ref{fig:fb-plots}.
%\begin{figure}[h]
%\begin{center}
%\includegraphics{fb-plots.pdf}
%\end{center}
%\caption{Left: $f_\mathfrak{b}$ as a function of $E_\mathrm{R}$ for $y_0=0$ and $\kappa=3$ (solid curve), and the asymptotes
%$\pm M|E_\mathrm{R}|^{2/3}$ (dashed curves).  Right:  the same for $y_0=3$ and $\kappa=3$; such a value of $y_0$ is well-within the region $\TR$ lying to the right of $\rectangle$ in the $y_0$-plane.}
%\label{fig:fb-plots}
%\end{figure}
Note that $E_\mathrm{R}(0,\kappa)=0$ holds for all $\kappa\in (-1,1)$, as was shown directly in Section~\ref{sec:Boutroux-curves-y0-zero}.  

By uniqueness, the solution obtained by the real-variable method based on the Intermediate Value Theorem agrees for sufficiently small real $y_0$ with that obtained by the application of the Implicit Function Theorem to the coupled system \eqref{eq:Boutroux} described in Section~\ref{sec:BoutrouxDomains}.  However, while the latter method fails at the boundary of the Boutroux domain $\rectangle$ where the spectral curve degenerates from class $\{1111\}$ to class $\{211\}$, the real $y_0$ method allows the continuation of the solution to larger real $y_0$.  Moreover, the real $y_0$ method makes clear that the mechanism of failure of the method of Section~\ref{sec:BoutrouxDomains} is not that the Jacobian of the equations \eqref{eq:Boutroux} vanishes, but rather that it blows up at the boundary of $\rectangle$.  

\subsubsection{Proof that $\TR$ is a Boutroux domain.  Abstract Stokes graphs of $\TR$}
\label{sec:AbstractStokesGraphsForTR}
It is now easy to prove that $\TR$ is a Boutroux domain.  Fix $\kappa\in (-1,1)$.  We choose any point $y_0=y_0^\#\in\TR\cap\mathbb{R}$
and use the Intermediate Value Theorem method of Section~\ref{sec:R-to-TR} to obtain the unique corresponding Boutroux curve $\mathcal{R}^\#$ with $E=E^\#=E_\mathrm{R}(y_0^\#,\kappa)\in\mathbb{R}$.  The curve $\mathcal{R}^\#$ is necessarily of class $\{1111\}$, because it is part of the definition of the domain $\TR$ that it is a component of the complement of the image in the $y_0$-plane of the critical set $\critclosure$ for the quadratic differential $\Phi'(t)^2\,\dd t^2$, where all degenerate Boutroux curves lie according to Section~\ref{sec:DegenerateStokesGraphs}.  So, taking $(y_0^\#,E^\#)$ as an initial condition, we apply the Implicit Function Theorem approach of Section~\ref{sec:BoutrouxDomains} to continue the solution of the Boutroux equations \eqref{eq:Boutroux} to the maximal Boutroux domain on which the Boutroux spectral curve is of class $\{1111\}$.  But by definition of $\TR$ once again, this maximal Boutroux domain coincides with $\TR$ itself.  The mapping $\TR\ni y_0\mapsto (E_\mathrm{R},E_\mathrm{I})\in\mathbb{R}^2$ obtained in this way coincides with the mapping $(\TR\cap\mathbb{R})\ni y_0\mapsto (E_\mathrm{R},0)\mapsto E_\mathrm{R}\in\mathbb{R}$ described in Section~\ref{sec:R-to-TR} by uniqueness.

To obtain the abstract Stokes graphs of $\TR$, it is sufficient to fix $\kappa\in (-1,1)$ and consider $y_0\in\TR\cap\mathbb{R}$.
The set $\TR\cap\mathbb{R}$ is a non-empty real interval $(y_0^\mathrm{e},y_0^\mathrm{c})$, where $y_0^\mathrm{e}$ is the real positive point on $\partial \rectangle$ (hence ``e'' for ``edge'') and $y_0^\mathrm{c}$ is the real positive corner point of $\TR$, i.e., the real positive root of $B(\cdot;\kappa)$.  The spectral curve is of class $\{211\}$ when $y_0=y_0^\mathrm{e}$, class $\{1111\}$ when $y_0\in (y_0^\mathrm{e},y_0^\mathrm{c})$, and class $\{31\}$ when $y_0=y_0^\mathrm{c}$, and in all cases since $y_0$ and $E=E_\mathrm{R}(y_0;\kappa)$ are real the spectral curve has Schwarz reflection symmetry in the real axis.  Finally, the coefficients of $P(z)$ are continuous functions of $y_0\in\mathbb{R}$, which implies continuity of the set of roots in $\mathbb{C}^4$ modulo permutations.  The spectral curve for $y_0=y_0^\mathrm{e}$ has the abstract Stokes graph matching either the first or third panel in Figure~\ref{fig:AbstractStokesGraphsRectangle-kappa0},
%for $\kappa\in (-1,1)$
% (resp., in Figure~\ref{fig:AbstractStokesGraphsRectangle-kappa3} for $\kappa>1$), 
so in light of Schwarz symmetry $P(z)$ has one real double root and a conjugate pair of simple roots.  The question to be addressed is:  given that the real double root of $P(z)=P(z^*)^*$ for $y_0=y_0^\mathrm{e}$ has to split into two simple roots for $y_0\in (y_0^\mathrm{e},y_0^\mathrm{c})$ because the Boutroux spectral curve is not degenerate (class $\{1111\}$), do those simple roots form a real pair or a complex-conjugate pair?  This question can be answered by considering the limit $y_0\uparrow y_0^\mathrm{c}$, in which the spectral curve degenerates to class $\{31\}$ and hence in light of Schwarz symmetry has exactly two roots, both real, one of which is triple.  It follows that the double root for $y_0=y_0^\mathrm{e}$ has to split into a purely real pair of roots since otherwise the spectral curve of class $\{1111\}$ for $y_0\in (y_0^\mathrm{e},y_0^\mathrm{c})$ would have no real roots, making the appearance of a triple real root as $y_0\uparrow y_0^\mathrm{c}$ impossible.  
From this reasoning, it is now clear that 
%Although the polynomial $P(z)$ has two real and two complex-conjugate roots for $y_0\in\mathbb{R}\cap \TR$ regardless of the value of $\kappa$, the Stokes graphs have different topology for $\kappa\in (-1,1)$ than for $\kappa>1$, a property inherited from the difference between the first and third panels of Figures~\ref{fig:AbstractStokesGraphsRectangle-kappa0} and \ref{fig:AbstractStokesGraphsRectangle-kappa3}, which show the abstract Stokes graphs at the real points  $y_0=\pm y_0^\mathrm{e}\in\partial \rectangle$ in the two disjoint ranges of $\kappa$.  Therefore, 
the abstract Stokes graphs for the opposite Boutroux domains $\pm \TR$ are as shown in the left two panels of 
Figure~\ref{fig:AbstractStokesGraphsTriangles}.  Examples of actual Stokes graphs corresponding to points in $\TR(\kappa)$ can be seen in Figure~\ref{fig:Trajectories-gO-k0}.
\begin{figure}[h]
\begin{center}
\includegraphics{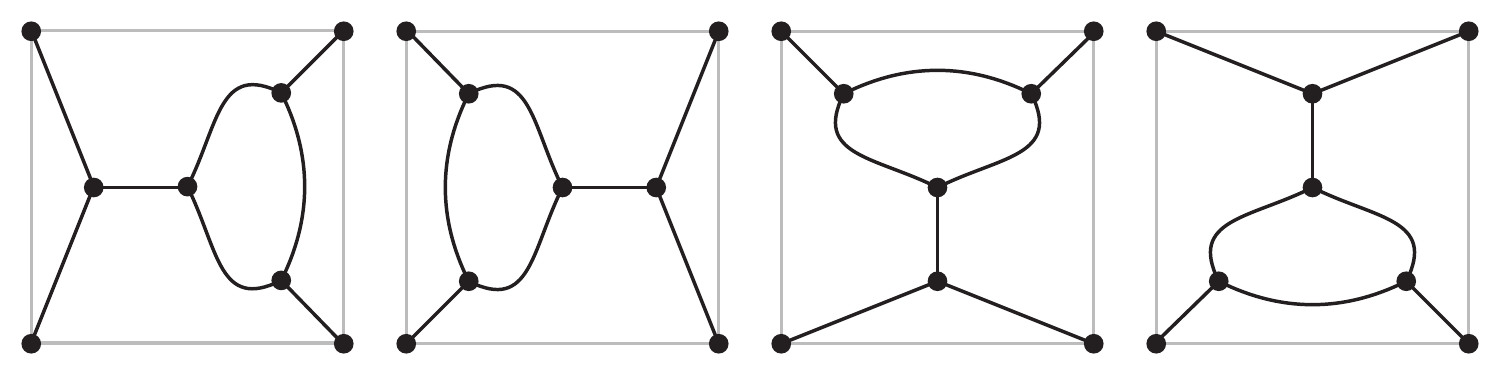}
\end{center}
\caption{Left two panels:  the abstract Stokes graphs for the Boutroux domains $\pm\TR$.  Right two panels:  the abstract Stokes graphs for the Boutroux domains $\pm\TI$.  Compare with Figure~\ref{fig:AbstractStokesGraphsRectangle-kappa0}.}
\label{fig:AbstractStokesGraphsTriangles}
\end{figure}

The value of $E$ as determined by the single-variable method of Section~\ref{sec:R-to-TR} varies continuously with $y_0\in\mathbb{R}$.  In fact, the smooth functions $y_0\mapsto E$ defined on $\rectangle(\kappa)$ and $\TR(\kappa)$ respectively both extend continuously from each side to the common boundary arc, on which the boundary values from each side agree.  The common boundary value of $E$ is exactly that which gives rise to a degenerate Boutroux curve with abstract Stokes graph
associated with the boundary arc.

%Figure~\ref{fig:AbstractStokesGraphsRealTriangles}.
%\begin{figure}[h]
%\begin{center}
%\includegraphics{AbstractStokesGraphsRealTriangles.pdf}
%\end{center}
%\caption{The abstract Stokes graphs for the Boutroux domains $\pm\TR$.
%% when $\kappa\in (-1,1)$ (left two panels) and when $\kappa>1$ (right two panels).
%}
%\label{fig:AbstractStokesGraphsRealTriangles}
%\end{figure}

%Note that once the continuation into the Boutroux domain $\TR$ along the real line has been achieved, we may continue the solution to all complex $y_0$ within $\TR$ by applying the implicit function theorem to the system \eqref{eq:Boutroux}, and that this solution will persist up to the boundary of $\TR$, as will the topology of the Stokes graph.  

\subsection{The Boutroux domain $\TI(\kappa)$}
\label{sec:TI}
Now let $\TI=\TI(\kappa)$ denote the maximal connected component of $\mathbb{C}\setminus\overline{\rectangle\cup\OkamotoExterior(\kappa)}$ that is entirely contained within the upper half $y_0$-plane (see the upper right-hand panel of Figure~\ref{fig:Domains}).  Again, $\TI$ is a simply connected domain whose boundary is a curvilinear triangle, the vertices of which are the vertices of an equilateral triangle, one of which lies on the positive imaginary axis.
We follow similar steps as in Section~\ref{sec:TR} to show that the domain $\TI$ is a Boutroux domain and to obtain its abstract Stokes graphs.  
%Unlike for $\TR$, the abstract Stokes graphs for $\TI$ when $\kappa\in (-1,1)$ and when $\kappa>1$ turn out to be the same.
%\subsubsection{Boutroux curves for $y_0\in\ii\mathbb{R}$.  Moving from $\rectangle$ into $\TI$}

Let $y_0=\ii w$ and $E=\ii E_\mathrm{I}$ both be purely imaginary.  Then it is clear that the polynomial $P(z)$ has Schwarz reflection symmetry about the imaginary axis in the $z$-plane:  $P(-z^*)=P(z)^*$.  Following a similar approach as in Section~\ref{sec:R-to-TR}, we find that one of the Boutroux conditions in \eqref{eq:Boutroux} is trivially satisfied, while the other gives a real relation between $w\in\mathbb{R}$ and $E_\mathrm{I}\in\mathbb{R}$ that can be solved uniquely for $E_\mathrm{I}$ by the Intermediate Value Theorem, yielding a continuous real map $w\mapsto E_\mathrm{I}$.  This allows us to find a unique Boutroux curve $\mathcal{R}$ of class $\{1111\}$ for each point $y_0\in \ii\mathbb{R}\cap \TI$.  Extending this solution of the Boutroux equations to all of $\TI$ by the method of Section~\ref{sec:BoutrouxDomains} we deduce that $\TI$ is indeed a Boutroux domain.  By uniqueness of both the real-variable and complex variable continuation methods and the symmetry $(y_0,\kappa,E)\mapsto (\ii y_0,-\kappa, -\ii E)$ of the Boutroux equations \eqref{eq:Boutroux}, it follows that the function $E(y_0;\kappa)$ defined on $\TI(\kappa)$ is related to $E(y_0;-\kappa)$ defined on $\TR(-\kappa)$ by $E(y_0;-\kappa)=\ii E(\ii y_0;\kappa)$ for all $\kappa\in (-1,1)$.  It then also follows that the function $y_0\mapsto E$ defined on $\TI(\kappa)$ matches that defined on $\rectangle(\kappa)$ when both are extended by continuity to their common boundary arc.  

Again considering $y_0$ to be a purely imaginary point in $\TI$, we can exploit the Schwarz symmetry of $P$ about the imaginary axis to obtain the 
%possible 
abstract Stokes 
%graphs 
graph
for the Boutroux domain $\TI$.  Indeed, at any point common to $\partial\rectangle$ and $\partial\TI$, the abstract Stokes graphs are as shown in the second and fourth panels of Figure~\ref{fig:AbstractStokesGraphsRectangle-kappa0}. 
%for $\kappa\in (-1,1)$ (resp., Figure~\ref{fig:AbstractStokesGraphsRectangle-kappa3} for $\kappa>1$). 
From these and the fact that at the positive imaginary corner point of $\TI$ the spectral curve is of class $\{31\}$, we can apply Schwarz symmetry to deduce the abstract Stokes graphs for the opposite Boutroux domains $\pm\TI$, which are shown in the right two panels of Figure~\ref{fig:AbstractStokesGraphsTriangles}.  See Figure~\ref{fig:Trajectories-gO-k0} for examples of actual Stokes graphs for $y_0\in\TI(\kappa)$.
%Significantly, there is now no (topological) distinction between the cases $\kappa\in (-1,1)$ and $\kappa>1$.  Indeed, the abstract Stokes graphs for both cases are shown in Figure~\ref{fig:AbstractStokesGraphsImaginaryTriangles}.
%\begin{figure}[h]
%\begin{center}
%\includegraphics{AbstractStokesGraphsImaginaryTriangles.pdf}
%\end{center}
%\caption{The abstract Stokes graphs for $\pm\TI$.  Unlike the case of $\pm\TR$, the Stokes graphs do not depend on whether $\kappa\in (-1,1)$ or $\kappa>1$ holds.}
%\label{fig:AbstractStokesGraphsImaginaryTriangles}
%\end{figure}

\begin{remark}
The method of Section~\ref{sec:R-to-TR} allows the unique determination of a Boutroux spectral curve not only on the parts of the real and imaginary axes in the bounded region $\mathbb{C}\setminus\OkamotoExterior(\kappa)$, but also in the exterior domain $\OkamotoExterior(\kappa)$.  The existence of a Boutroux spectral curve along the axes in the exterior domain $\OkamotoExterior(\kappa)$ is of course consistent with the result that the class $\{211\}$ spectral curves for such $y=y_0$ coincide with degenerate Boutroux curves (these spectral curves are discussed in Section~\ref{sec:ExteriorAxes}).  Indeed, such $y_0$ lie on the images of the critical v-trajectories $(a(w),0)$ and $(0,b(w))$ within $S_\mathrm{gO}$.  Therefore for these values of $y_0$, the corresponding value of $E$ as determined from the Intermediate Value Theorem as described in Section~\ref{sec:R-to-TR} coincides with the value given explicitly in terms of $\gamma=U_{0,\mathrm{gO}}(y_0;\kappa)\sim -\tfrac{2}{3}y_0$ and $y_0$ by \eqref{eq:alpha-beta-E} in which $y$ is replaced with $y_0$ (see also \eqref{eq:E-for-double-roots}).
\end{remark}

%\textcolor{red}{This apparently gives a Boutroux curve for each $y\in\mathbb{R}$ and each $c<-1$, and allows us to continue the solution from the real cross-section of the interior ``rectangle'' into two of the four adjoining ``triangles'' by moving along the real axis.  But it also seems to allow us to pass outside of the triangles and get a Boutroux curve for real $y$ in the exterior region, where we know that a genus-zero ansatz solves the problem.  However, we note that it has been shown that there is indeed such a Boutroux curve for each sufficiently large real $y$; however the curve is degenerate. }

\section{Asymptotic analysis of $\mathbf{M}(z)$ in the relevant Boutroux domains for the gO and gH cases}
\label{sec:G1}
As mentioned at the beginning of Section~\ref{sec:Boutroux}, for $T=|\Theta_0|\gg 1$, $s=\mathrm{sgn}(\Theta_0)=\pm 1$ we represent $x$ in the form (cf.\@ \eqref{eq:BasicScalings}) $x=T^{1/2}y_0+T^{-1/2}\zeta$, where we assume that $y_0\in\rectangle(\kappa)\cup\TR(\kappa)\cup\TI(\kappa)$ and $s=\pm 1$ (gO case) or $y_0\in\rectangle(\kappa)$ and $s=1$ (gH case), and that $\zeta\in\mathbb{C}$.  Recalling from Remark~\ref{rem:other-parameters} that the parameters $\Theta_0$ and $\Theta_\infty$ are associated with the function $u(x)=u^{[3]}_\mathrm{F}(x;m,n)$ for $\mathrm{F}=\mathrm{gH}$ or $\mathrm{F}=\mathrm{gO}$, 
we have $\kappa:=-\Theta_\infty/T\in (-1,1)$ differing from a limiting value in the same interval by $\bo(T^{-1})$.  The matrix $\mathbf{M}(z)=\mathbf{M}^{(T,s,\kappa)}(z;y)=\mathbf{M}^{(T,s,\kappa)}(z;y_0+T^{-1}\zeta)$ depends on $(y_0,\zeta)$ only through the linear combination $y=y_0+T^{-1}\zeta$, but later $(y_0,\zeta)$ will admit interpretation as independent variables.  

Indeed, the variables $(y_0,\zeta)$ immediately play distinguished roles, as the first step of our analysis of $\mathbf{M}(z)$ is to construct the Stokes graph in the $z$-plane determined by 
%$\kappa\in (-1,1)\cup (1,+\infty)$ 
$\kappa\in (-1,1)$
and the value of $y_0$ alone.  As explained in Sections~\ref{sec:AbstractStokesGraphsForRectangle}, \ref{sec:AbstractStokesGraphsForTR}, and \ref{sec:TI}, the abstract Stokes graph is more robust, depending only on the Boutroux domain containing $y_0$; see the right-hand panel of Figure~\ref{fig:a-b-cycles} for $y_0\in\rectangle(\kappa)$, the two left-hand panels of Figure~\ref{fig:AbstractStokesGraphsTriangles} for $y_0\in\pm\TR(\kappa)$, and the two right-hand panels of Figure~\ref{fig:AbstractStokesGraphsTriangles} for $y_0\in\pm\TI(\kappa)$.
%depending on the value of $\kappa$, there are two possible abstract Stokes graphs when $y_0\in\rectangle$ (see Figure~\ref{fig:ASGs-rectangle}), two more when $y_0\in\TR$ (see Figure~\ref{fig:AbstractStokesGraphsRealTriangles}), and one more when $y_0\in\TI$ (see Figure~\ref{fig:AbstractStokesGraphsImaginaryTriangles}).  
Numerically generated Stokes graphs for representative values of $y_0\in\rectangle(\kappa)\cup\TR(\kappa)\cup\TI(\kappa)$ and $\kappa=0$ are shown with black and orange arcs in the left-hand panels of Figures~\ref{fig:SignsContourLenses-y0-kappa0-s1}--\ref{fig:SignsContourLenses-y1p3i-kappa0-sm1} that can be found in Appendix~\ref{app:DiagramsAndTables}.  In these plots, the union of black and orange arcs form the same Stokes graph for both signs of $s=\pm 1$.  However, the orange arcs are distinguished as the branch cuts for $R(z)$ and hence $h'(z)$, and these necessarily differ for $s=1$ and $s=-1$.  As 
was shown in Section~\ref{sec:StokesGraphsForBoutrouxCurves}, as a consequence of the fact that $y_0$ lies in a Boutroux domain $\mathcal{B}=\rectangle(\kappa),\TR(\kappa),\TI(\kappa)$ and of the choice that the branch cuts of $h'(z)$ are taken to be arcs of the Stokes graph, $\mathrm{Re}(h(z))=\mathrm{Re}(h^{(s,\kappa)}(z;y_0))$ is well-defined as a continuous function of $z\in\mathbb{C}\setminus\{0\}$ that is harmonic except on the branch cuts of $h'(z)$ (orange arcs) and that vanishes exactly on the Stokes graph.  Hence it has a well-defined sign on each connected component of the complement of the Stokes graph, and these signs are indicated in all plots in the Appendix~\ref{app:DiagramsAndTables} with white for positive and shading for negative. 

Also indicated on the left-hand panels of Figures~\ref{fig:SignsContourLenses-y0-kappa0-s1}--\ref{fig:SignsContourLenses-y1p3i-kappa0-sm1} are two contours, $\ell_1$ and $\ell_2$ (in some cases $\ell_2$ is a union of two arcs), and invoking the Boutroux conditions \eqref{eq:Boutroux} we use these arcs to define two real constants, $R_1$ and $R_2$, by the formula
\eq
R_j:=-\ii\int_{\ell_j}h'(z)\,\dd z=\frac{1}{4}\ii\int_{\ell_j}\frac{R(z)}{z}\,\dd z\in\mathbb{R},\quad j=1,2.
\label{eq:BoutrouxConstants}
\endeq 

The next step is to place the jump contour $\Sigma$ appropriately relative to the Stokes graph in the $z$-plane determined by $\kappa$ and $y_0$.  We use 
\begin{itemize}
\item In the gO case:
\begin{itemize}
\item the basic configuration of $\Sigma$ (see Figure~\ref{fig:PIV-Sigma}, in the $z$-plane instead of the $\lambda$-plane) when $y_0\in\rectangle$,
\item the ``leftward'' deformation of $\Sigma$ (see the left-hand panel of Figure~\ref{fig:PIV-Sigma-Deform}) when $y_0\in\TR$, and
\item the ``downward'' deformation of $\Sigma$ (see the right-hand panel of Figure~\ref{fig:PIV-Sigma-Deform}) when $y_0\in\TI$;
\end{itemize}
\item In the gH case:  the simplified contour $\Sigma=\Sigma_\mathrm{gH}$ described in Remark~\ref{rem:gH-simplification} (see Figure~\ref{fig:PIV-Sigma-Hermite}, in the $z$-plane instead of the $\lambda$-plane, and denoting the deformed closed-loop component of $\Sigma_\mathrm{gH}$ by $C$).
\end{itemize}
We then arrange the arcs of the jump contour $\Sigma$ relative to the Stokes graph as shown in the  right-hand panels of Figures~\ref{fig:SignsContourLenses-y0-kappa0-s1}--\ref{fig:SignsContourLenses-y1p3i-kappa0-sm1}.  In these plots, an arc of $\Sigma$ is frequently split into sub-arcs that belong to different critical v-trajectories in the Stokes graph, in which case we use superscripts to distinguish the different sub-arcs that carry the same jump matrix for $\mathbf{M}(z)$.  In particular for the gO case, subarcs $\Sigma_j^k$ of $\Sigma_j$, $j=1,2,3,4$ will always be placed on $\partial\circledomain$ and the arc $\Sigma_\mathrm{c}$ will always be placed on another arc of the Stokes graph.  In both the gO and gH cases, the arcs of $\Sigma$ carrying triangular jump matrices satisfy the principle that those carrying lower-triangular (resp., upper-triangular) matrices should be placed such that the possibly non-strict inequality $\mathrm{Re}(h(z))\le 0$ (resp., $\mathrm{Re}(h(z))\ge 0$) holds, in order to avoid exponential growth of the off-diagonal matrix element.  Finally, the arc $\Sigma_0$ always carries a diagonal jump matrix, and it may be placed arbitrarily within the circle domain $\circledomain$ once its endpoints have been fixed by prior considerations (one of the endpoints is $z=0$).

\subsection{Steepest descent}
\label{sec:Genus-One-Steepest-Descent}
While $\mathrm{Re}(h(z))$ is well-defined as a continuous function on the $z$-plane, $h(z)$ itself has discontinuities along certain cuts due to the monodromy of $\mathrm{Im}(h(z))$ about (i) the branch cuts of $h'(z)$ as measured by the real constants $R_j$, $j=1,2$ defined in \eqref{eq:BoutrouxConstants}, and (ii) the poles of $h'(z)$ at $z=0$ and $z=\infty$ as measured by the real residues $-\kappa$ and $-s$, respectively. We will explain how we determine one or more purely imaginary integration constants to fix $h(z)$ later in Section~\ref{sec:GenusOneDefineh}.  However, we can say now that $h(z)$ will certainly be analytic except on some arcs of the jump contour $\Sigma$ for $\mathbf{M}(z)$ as already determined.  Therefore 
%for some arcs of $\partial\circledomain$, the arcs $\Sigma_0$ and $\Sigma_\mathrm{c}$ (if present), and either the arc $\Sigma_{4,3}$ (for $y_0\in\rectangle\cup\TR$) or the arc $\Sigma_{4,1}$ (for $y_0\in\TI$).  In all cases, these arcs are part of the jump contour $\Sigma$, so 
$g(z)=g^{(s,\kappa)}(z;y_0):=h^{(s,\kappa)}(z;y_0)+\phi(z;y_0)$ is analytic for $z\in\mathbb{C}\setminus\Sigma$
%.   Therefore 
so
we may use the formula \eqref{eq:NfromM} to transform $\mathbf{M}(z)$ into $\mathbf{N}(z)$.  Because $\mathbf{M}(z)=\mathbf{M}^{(T,s,\kappa)}(z;y_0+T^{-1}\zeta)$ depends on $y=y_0+T^{-1}\zeta$ while $g$ is independent of $\zeta$, we write $\mathbf{N}(z)=\mathbf{N}^{(T,s,\kappa)}(z;y_0,\zeta)$ to distinguish the now-independent roles of the parameters $y_0\in\rectangle(\kappa)\cup\TR(\kappa)\cup\TI(\kappa)$ and $\zeta\in\mathbb{C}$.

The jump conditions for $\mathbf{M}(z)$ take the form $\mathbf{M}_+(z)=\mathbf{M}_-(z)\ee^{T\phi(z;y)\sigma_3}\mathbf{V}\ee^{-T\phi(z;y)\sigma_3}$ where $\mathbf{V}$ is a different constant matrix on different arcs of $\Sigma$.  Noting that $T\phi(z;y)=T\phi(z;y_0)+\tfrac{1}{2}\zeta z$, it then follows that the corresponding jump condition for $\mathbf{N}(z)$ on a given arc $A\subset\Sigma$ can be written as 
\eq
\mathbf{N}_+(z)=\mathbf{N}_-(z)\ee^{\tfrac{1}{2}\zeta z\sigma_3}\ee^{-Th_-(z)\sigma_3}\mathbf{V}\ee^{Th_+(z)\sigma_3}\ee^{-\tfrac{1}{2}\zeta z\sigma_3},\quad z\in A\subset\Sigma.
\endeq
Here $h_\pm(z)$ denote the boundary values taken on the arc $A$ by $h(z)$.  We then transform $\mathbf{N}(z)$ into $\mathbf{O}(z)$ by the following steepest-descent procedure, in which we ``open lenses'' of the jump contour about those arcs of $\Sigma$ lying on the Stokes graph, with the sole exception of $\Sigma_\mathrm{c}$ when $y_0\in\TR(\kappa)\cup\TI(\kappa)$.  We first factor the constant ``core'' jump matrix $\mathbf{V}$ for each such arc $A$ into a product $\mathbf{V}^-\mathbf{V}^0\mathbf{V}^+$ of three unit determinant constant factors using an appropriate choice among the identities \eqref{eq:general-factorizations}.  The factor $\mathbf{V}^+$ (resp., $\mathbf{V}^-$) will be associated with the region of $\mathbb{C}\setminus\Sigma$ lying immediately to the left (resp., right) of the arc $A$.  We require $\mathbf{V}^\pm$ to be lower (upper) triangular if the inequality $\mathrm{Re}(h(z))<0$ ($\mathrm{Re}(h(z))>0$) holds on the corresponding region of the $z$-plane; this criterion then determines exactly which of the identities \eqref{eq:general-factorizations} is to be used in each case.  
For the gO family, we will need to apply this procedure to the four ``core'' jump matrices 
%For arcs $A$ lying on $C=\partial\circledomain$, the ``core'' jump matrix is 
$\mathbf{V}=\mathbf{V}_j$, $j=1,\dots,4$, as defined by \eqref{eq:gO-connection}
for which the factorizations in \eqref{eq:general-factorizations} take the form
% in Riemann-Hilbert Problem~\ref{rhp:00}, and we use the following factorizations, all of which are special cases of \eqref{eq:general-factorizations} and are written with the notation \eqref{eq:matrix-factors-notation}: 
\eq
\begin{split}
\mathbf{V}_1 &= \mathbf{L}(2\ee^{\frac{\ii\pi}{6}})\mathbf{D}(\tfrac{1}{\sqrt{3}})\mathbf{U}(-\tfrac{\sqrt{3}}{2})\quad \text{(``LDU'')}\\
&=\mathbf{L}(2\ee^{\frac{5\ii\pi}{6}})\mathbf{T}(2)\mathbf{L}(-\tfrac{2}{\sqrt{3}})\quad\text{(``LTL'')}\\
&=\mathbf{U}(\tfrac{1}{2}\ee^{-\frac{5\ii\pi}{6}})\mathbf{D}(\ee^{\frac{\ii\pi}{6}}) \mathbf{L}(\tfrac{2}{\sqrt{3}}\ee^{\frac{\ii\pi}{3}})\quad\text{(``UDL'')}\\
&=\mathbf{U}(\tfrac{1}{2}\ee^{-\frac{\ii\pi}{6}})\mathbf{T}(\tfrac{2}{\sqrt{3}}\ee^{\frac{\ii\pi}{6}})\mathbf{U}(\tfrac{\sqrt{3}}{2}\ee^{-\frac{\ii\pi}{3}})\quad\text{(``UTU'')},
\end{split}
\endeq
\eq
\begin{split}
\mathbf{V}_2=\mathbf{V}_1^{*-1}&=\mathbf{L}(\tfrac{2}{\sqrt{3}}\ee^{\frac{2\ii\pi}{3}})\mathbf{D}(\ee^{\frac{\ii\pi}{6}})\mathbf{U}(\tfrac{1}{2}\ee^{-\frac{\ii\pi}{6}})\quad\text{(``LDU'')}\\
&=\mathbf{L}(\tfrac{2}{\sqrt{3}})\mathbf{T}(-2)\mathbf{L}(2\ee^{\frac{\ii\pi}{6}})\quad\text{(``LTL'')}\\
&=\mathbf{U}(\tfrac{\sqrt{3}}{2})\mathbf{D}(\sqrt{3})\mathbf{L}(2\ee^{\frac{5\ii\pi}{6}})\quad\text{(``UDL'')}\\
&=\mathbf{U}(\tfrac{\sqrt{3}}{2}\ee^{-\frac{2\ii\pi}{3}})\mathbf{T}(\tfrac{2}{\sqrt{3}}\ee^{\frac{5\ii\pi}{6}})\mathbf{U}(\tfrac{1}{2}\ee^{-\frac{5\ii\pi}{6}})\quad\text{(``UTU'')},
\end{split}
\endeq
\eq
\begin{split}
\mathbf{V}_3&=\mathbf{L}(2\ee^{\frac{\ii\pi}{6}})\mathbf{D}(\tfrac{1}{\sqrt{3}}\ee^{-\frac{\ii\pi}{3}})\mathbf{U}(\tfrac{\sqrt{3}}{2}\ee^{-\frac{\ii\pi}{3}})\quad\text{(``LDU'')}\\
&=\mathbf{L}(2\ee^{\frac{5\ii\pi}{6}})\mathbf{T}(2\ee^{-\frac{\ii\pi}{3}})\mathbf{L}(\tfrac{2}{\sqrt{3}}\ee^{\frac{\ii\pi}{3}})\quad\text{(``LTL'')}\\
&=\mathbf{U}(\tfrac{1}{2}\ee^{-\frac{5\ii\pi}{6}})\mathbf{D}(\ee^{-\frac{\ii\pi}{6}})\mathbf{L}(\tfrac{2}{\sqrt{3}}\ee^{-\frac{\ii\pi}{3}})\quad\text{(``UDL'')}\\
&=\mathbf{U}(\tfrac{1}{2}\ee^{-\frac{\ii\pi}{6}})\mathbf{T}(\tfrac{2}{\sqrt{3}}\ee^{-\frac{\ii\pi}{6}})\mathbf{U}(\tfrac{\sqrt{3}}{2}\ee^{\frac{\ii\pi}{3}})\quad\text{(``UTU'')},
\end{split}
\endeq
\eq
\begin{split}
\mathbf{V}_4=\mathbf{V}_3^{*-1}&=\mathbf{L}(\tfrac{2}{\sqrt{3}}\ee^{-\frac{2\ii\pi}{3}})\mathbf{D}(\ee^{-\frac{\ii\pi}{6}})\mathbf{U}(\tfrac{1}{2}\ee^{-\frac{\ii\pi}{6}})\quad\text{(``LDU'')}\\
&=\mathbf{L}(\tfrac{2}{\sqrt{3}}\ee^{\frac{2\ii\pi}{3}})\mathbf{T}(2\ee^{-\frac{2\ii\pi}{3}})\mathbf{L}(2\ee^{\frac{\ii\pi}{6}})\quad\text{(``LTL'')}\\
&=\mathbf{U}(\tfrac{\sqrt{3}}{2}\ee^{-\frac{2\ii\pi}{3}})\mathbf{D}(\sqrt{3}\ee^{-\frac{\ii\pi}{3}})\mathbf{L}(2\ee^{\frac{5\ii\pi}{6}})\quad\text{(``UDL'')}\\
&=\mathbf{U}(\tfrac{\sqrt{3}}{2}\ee^{\frac{2\ii\pi}{3}})\mathbf{T}(\tfrac{2}{\sqrt{3}}\ee^{-\frac{5\ii\pi}{6}})\mathbf{U}(\tfrac{1}{2}\ee^{-\frac{5\ii\pi}{6}})\quad\text{(``UTU'')}.
\end{split}
\endeq
For the gH family, the only ``core'' jump matrix that requires any factoring is the matrix $\mathbf{L}(1)$ and just the ``UTU'' factorization suffices:
\eq
\mathbf{L}(1)=\mathbf{U}(1)\mathbf{T}(1)\mathbf{U}(1).
\endeq
%Also, for the arcs $A\subset\Sigma$, $A\neq\Sigma_\mathrm{c}$, lying on the part of the Stokes graph disjoint from $C=\partial\circledomain$ we use the following ``LTL'' factorization:
%\eq
%\mathbf{U}(-\tfrac{1}{2}\ii)=\mathbf{L}(2\ii)\mathbf{T}(-2\ii)\mathbf{L}(2\ii).
%\endeq

Now let $A^+$ (resp., $A^-$) denote an arc with the same endpoints and orientation as $A$ but lying in the region to the left (resp., right) of $A$; thus $A^\pm$ form a ``lens'' about the central arc $A$.  Then, we set
\eq
\mathbf{O}(z):=\mathbf{N}(z)\ee^{\tfrac{1}{2}\zeta z\sigma_3}\ee^{-Th(z)\sigma_3}(\mathbf{V}^\pm)^{\mp 1}\ee^{Th(z)\sigma_3}\ee^{-\tfrac{1}{2}\zeta z\sigma_3},\quad \text{for $z$ between $A$ and $A^\pm$}.
\endeq
Repeating this substitution for each arc $A\subset\Sigma$, $A\neq\Sigma_\mathrm{c}$, lying on the Stokes graph and elsewhere defining $\mathbf{O}(z):=\mathbf{N}(z)$, we arrive at an equivalent unknown matrix $\mathbf{O}(z)=\mathbf{O}^{(T,s,\kappa)}(z;y_0,\zeta)$.  The matrix function $z\mapsto\mathbf{O}(z)$ is analytic except on the contour $\Sigma$ augmented with the lens boundaries $A^\pm$ and omitting $\Sigma_0$ and, in the gH case, $\Sigma_{4,3}$ (the omitted arcs are removed already by the substitution $\mathbf{M}(z)\mapsto\mathbf{N}(z)$).  Also $\mathbf{O}(z)$ satisfies the normalization condition $\mathbf{O}(z)\to\mathbb{I}$ as $z\to\infty$.  The jump conditions for $\mathbf{O}(z)$ are illustrated for each case in Figures~\ref{fig:JustLenses-y0-kappa0-s1}--\ref{fig:JustLenses-y1p3i-kappa0-sm1}, in which we specify the matrix $\mathbf{W}$ for each arc such that $\widetilde{\mathbf{O}}_+(z)=\widetilde{\mathbf{O}}_-(z)\mathbf{W}$, where
\eq
\widetilde{\mathbf{O}}(z):=\mathbf{O}(z)\ee^{\tfrac{1}{2}\zeta z\sigma_3}\ee^{-Th(z)\sigma_3}.
\endeq
Note that $\mathbf{W}$ is a ``core'' jump matrix for $\mathbf{O}(z)$ on an arc of the modified jump contour in the same way that $\mathbf{V}$ is a ``core'' jump matrix for $\mathbf{N}(z)$ on an arc of $\Sigma$.

Some further simplification of the jump contour for $\mathbf{O}(z)$ is easily performed.  Indeed, one may observe that in any situation that two lenses share a common endpoint $z_0\in\partial\circledomain$ that is not one of the four distinguished points $z=\alpha,\beta,\gamma,\delta$, the (two or three) contours approaching $z_0$ either from within or from without $\circledomain$ can be fused together into a single arc across which $\mathbf{O}(z)$ experiences no jump at all.  One then checks further that the remaining jumps for $\mathbf{O}(z)$ on the two remaining arcs of $\Sigma$ that meet at $z_0$ are consistent.  
\begin{example}
In the configuration for the gO family with $y_0\in\rectangle(\kappa)$ and $s=1$ as depicted in Figure~\ref{fig:SignsContourLenses-y0-kappa0-s1}, we can regard the jump contour for $\mathbf{O}(z)$ as having just one lens joining the pair of points $z=\alpha,\beta$ and consisting of 
\begin{itemize}
\item 
an arc inside $\circledomain$ from $\alpha$ to $\beta$ carrying the ``core'' jump matrix $\mathbf{W}=\mathbf{U}(\tfrac{\sqrt{3}}{2}\ee^{\frac{\ii\pi}{3}})$, 
\item
an arc along $\partial \circledomain$ from $\alpha$ to $\beta$ carrying the ``core'' jump matrix $\mathbf{W}=\mathbf{T}(\tfrac{2}{\sqrt{3}}\ee^{-\frac{\ii\pi}{6}})$, and 
\item
two arcs outside $\circledomain$ from an arbitrary point on $\Sigma_{2,3}$ away from $\partial\circledomain$ to $\alpha$ and to $\beta$ carrying ``core'' jump matrices $\mathbf{W}=\mathbf{U}(\tfrac{1}{2}\ee^{-\frac{5\ii\pi}{6}})$ and $\mathbf{W}=\mathbf{U}(\tfrac{1}{2}\ee^{-\frac{\ii\pi}{6}})$ respectively.  
\end{itemize}
Completely analogous deformations produce single lenses joining the pairs $z=\alpha,\delta$ and $z=\gamma,\delta$.  With the additional use of jump identities for the function $h(z)$ near the junction point between $\Sigma_{4,3}$ and $\Sigma_0$ (as shown for the case at hand in the left-hand panel of Figure~\ref{fig:hJAO-y0-kappa0-s1}), one achieves a similar result and obtains a single lens connecting $z=\beta,\gamma$, but the simplification occurs only for the matrix $\mathbf{O}(z)$ and not also $\widetilde{\mathbf{O}}(z)$, and it hinges on $sT=\Theta_0$, $T\kappa=-\Theta_\infty$, and the gO lattice conditions 
%$\Theta_0=\tfrac{1}{6}-\tfrac{1}{2}m$, and $\Theta_\infty=\tfrac{1}{2}-\tfrac{1}{2}m-n$ for $(m,n)\in\mathbb{Z}^2$.
$(\Theta_0,\Theta_\infty)\in\Lambda_\mathrm{gO}$ (see \eqref{eq:gO-lattice}).
\end{example}
In each case, the fact that such simplification is possible stems from the fact that the jump conditions of Riemann-Hilbert Problem~\ref{rhp:general} are consistent at all self-intersection points.  By analyticity of the exponent function $\lambda^2+2x\lambda$, this consistency makes the locations of these junction points somewhat arbitrary.

\subsection{Specification of $h(z)$}
\label{sec:GenusOneDefineh}
We now resolve all remaining ambiguity about the function $h(z)$.  As mentioned at the beginning of Section~\ref{sec:Genus-One-Steepest-Descent}, since $h'(z)$ is well-defined with branch cuts on the orange arcs of the Stokes graph as shown in the left-hand panels of Figures~\ref{fig:SignsContourLenses-y0-kappa0-s1}--\ref{fig:SignsContourLenses-y1p3i-kappa0-sm1}, all that remains to fully specify $h(z)$ is to complete its jump contour with additional arcs to handle the monodromy about the poles and then to give values for one or more integration constants.  The real parts of these constants are determined so that the level set $\mathrm{Re}(h(z))=0$ coincides with the Stokes graph determined by 
%$\kappa\in (-1,1)\cup (1,+\infty)$ 
$\kappa\in (-1,1)$
and $y_0\in\rectangle(\kappa)\cup\TR(\kappa)\cup\TI(\kappa)$.  
We complete the specification of $h(z)$ first for the following three gO cases:  (a) $y_0\in\rectangle(\kappa)$,  $s=1$, (b) $y_0\in\TR(\kappa)$, $s=1$, and (c) $y_0\in\TI(\kappa)$, $s=-1$.
%\begin{itemize}
%\item $y_0\in\rectangle\cup\TR$ and $s=1$ (for both ranges $-1<\kappa<1$ and $\kappa>1$, for four total cases), or
%\item $y_0\in\TI$ and $s=-1$ (for all $\kappa\in (-1,1)\cup (1,+\infty)$, hence just one case).
%\end{itemize}
We take additional cuts as shown for cases (a), (b), and (c) respectively in the left-hand panels of Figures~\ref{fig:hJAO-y0-kappa0-s1}, 
%\ref{fig:hJumpsAndOuter-y0-kappa1p05-s1}, 
\ref{fig:hJAO-y1p3-kappa0-s1}, 
%\ref{fig:hJumpsAndOuter-y2p025-kappa1p05-s1}, 
and \ref{fig:hJAO-y1p3i-kappa0-sm1} so that $h(z)$ is analytic in a simply connected domain $\Omega_h$, and therefore just one integration constant needs to be determined.  We fix it by setting
\eq
h(z)=\int_\delta^z h'(w)\,\dd w -\frac{1}{2}\ii R_1,\quad z\in\Omega_h
\label{eq:Genus-One-h-define}
\endeq
where the value of the integral is independent of any path of integration taken in the domain $\Omega_h$, 
and where $R_1\in\mathbb{R}$ is defined in \eqref{eq:BoutrouxConstants}.  Using the fact that $h'(z)$ has residues of $-s$ and $-\kappa$ at $z=0$ and $z=\infty$ respectively, and again taking note of \eqref{eq:BoutrouxConstants} referring to the location of the integration contours $\ell_1$ and $\ell_2$ as shown on the Stokes graph plots in the left-hand panels of Figures~\ref{fig:SignsContourLenses-y0-kappa0-s1}--\ref{fig:SignsContourLenses-y1p3i-kappa0-sm1}, one checks that the function $h(z)$ defined by \eqref{eq:Genus-One-h-define} satisfies the jump conditions on the sum and difference of boundary values as shown in the left-hand panels of Figures~\ref{fig:hJAO-y0-kappa0-s1}, 
%\ref{fig:hJumpsAndOuter-y0-kappa1p05-s1}, 
\ref{fig:hJAO-y1p3-kappa0-s1}, 
%\ref{fig:hJumpsAndOuter-y2p025-kappa1p05-s1}, 
and \ref{fig:hJAO-y1p3i-kappa0-sm1}.  To define $h(z)$ in the 
%remaining five cases, 
three remaining gO cases, (d) $y_0\in\rectangle(\kappa)$, $s=-1$, (e) $y_0\in\TR(\kappa)$, $s=-1$, and (f) $y_0\in\TI(\kappa)$, $s=1$,
we note that this requires changing the sign of $s$ from cases (a), (b), and (c) respectively, which also means that we change the sign of $h'(z)$ in the circle domain $\circledomain$ while leaving $h'(z)$  unchanged in the exterior of $\circledomain$.  Therefore it seems natural to also define $h(z)$ by starting with the formula \eqref{eq:Genus-One-h-define} and simply changing the sign of $h(z)$ within $\circledomain$.  We follow this approach and proceed with the implied choice of integration constants.  
Finally, for the only gH case ($y_0\in\rectangle(\kappa)$ and $s=1$), we assume that $h(z)$ is defined exactly as for the gO case with $y_0\in\rectangle(\kappa)$ and $s=1$.  It is then straightforward to verify the jump conditions satisfied by $h(z)$ in these four remaining cases as shown in the left-hand panels of Figures~\ref{fig:hJAO-y0-kappa0-s1-gH}, \ref{fig:hJAO-y0-kappa0-sm1}, 
%\ref{fig:hJumpsAndOuter-y0-kappa1p05-sm1}, 
\ref{fig:hJAO-y1p3-kappa0-sm1}, 
%\ref{fig:hJumpsAndOuter-y2p025-kappa1p05-sm1}, 
and \ref{fig:hJAO-y1p3i-kappa0-s1}.

In general, the constants $R_1$ and $R_2$ defined by \eqref{eq:BoutrouxConstants} are independent.  However, if $y_0\in(\rectangle(\kappa)\cup\TR(\kappa))\cap\mathbb{R}$, then with the indicated choice of integration constant we recover a Schwarz symmetry:  $h'(z^*)=h'(z)^*$.  Considering the difference of loop integrals of $h'(z)$ around $z=0,\infty$ computed by residues on one hand and by \eqref{eq:BoutrouxConstants} on the other, this symmetry in turn implies the identity
\eq
R_2=\tfrac{1}{2}\pi (1-\kappa),\quad y_0\in\rectangle(\kappa)\cap\mathbb{R},\quad\kappa\in (-1,1),
%\cup (1,+\infty),
\label{eq:R2-rectangle-y0-real}
\endeq
and, by a simpler computation,
\eq
R_2=0,\quad y_0\in\TR(\kappa)\cap\mathbb{R},\quad \kappa\in (-1,1).
%\cup (1,+\infty).
\label{eq:R2-TR-y0-real}
\endeq
Similarly, if $y_0\in(\rectangle(\kappa)\cup\TI(\kappa))\cap\ii\mathbb{R}$, then $h'(-z^*)=h'(z)^*$, which implies the identities
\eq
R_1=-\tfrac{1}{2}\pi(1+\kappa),\quad y_0\in\rectangle(\kappa)\cap\ii\mathbb{R},\quad \kappa\in (-1,1),\quad\text{and}
\label{eq:R1-rectangle-y0-imaginary}
\endeq
\eq
R_2=0,\quad y_0\in\TI(\kappa)\cap\ii\mathbb{R},\quad\kappa\in (-1,1).
%\cup (1,+\infty).
\label{eq:R2-TI-y0-imaginary}
\endeq

\subsection{Parametrix construction}
\label{sec:G1-parametrix}
We now show how to build an approximation of the matrix function $z\mapsto\mathbf{O}(z)$ the accuracy of which can be controlled.  
With $h(z)$ defined precisely as described in Section~\ref{sec:GenusOneDefineh}, we observe that due to the position of the jump contour for $\mathbf{O}(z)$ relative to the sign chart of $\mathrm{Re}(h(z))$, the jump matrix for $\mathbf{O}(z)$ is an exponentially small perturbation of the identity matrix for $T\gg 1$ wherever the jump matrix is upper or lower triangular.  

\subsubsection{Outer parametrix}
If we simply neglect these jumps and use known information about the boundary values of $h(z)$ in the remaining diagonal and off-diagonal jump matrices, we arrive at a modified Riemann-Hilbert problem for a formal approximation called the \emph{outer parametrix}.  More precisely, the outer parametrix
$\dot{\mathbf{O}}^\mathrm{out}(z)$ is defined as the matrix analytic in the complement of the jump contour shown in the right-hand panels of Figures~\ref{fig:hJAO-y0-kappa0-s1}--\ref{fig:hJAO-y1p3i-kappa0-sm1}, with the indicated jump matrix on each arc, normalized to the identity as $z\to\infty$, and continuous up to the jump contour with the exception of the four points $z=\alpha,\beta,\gamma,\delta$ at each of which a negative one-fourth root divergence is admitted to account for the discontinuity of the jump matrix.    

While the details are different in each case, the conditions characterizing the outer parametrix can be easily mapped to a single universal form.  
%For this purpose, we need to first define a domain $\specialdomain$ of the $z$-plane
%in the two cases for which $y_0\in\TR$ and $\kappa>1$ (for both signs $s=\pm 1$).  Referring to the upper right-hand panels of Figures~\ref{fig:hJumpsAndOuter-y2p025-kappa1p05-s1}--\ref{fig:hJumpsAndOuter-y2p025-kappa1p05-sm1}, we let $B_2$ denote an oriented arc from $z=\gamma$ to $z=\delta$ lying otherwise to the left of the indicated jump contour for $\dot{\mathbf{O}}^\mathrm{out}(z)$.  Then we define $\specialdomain$ as the region bounded by $B_2$ and the jump contour.  
%With this definition, we then 
Indeed, we will
define a new unknown $\dot{\mathbf{P}}^\mathrm{out}(z)$ in terms of $\dot{\mathbf{O}}^\mathrm{out}(z)$ in different regions of the $z$-plane according to Table~\ref{tab:outer-uniformize}.
\begin{table}[h]
\caption{The relation between $\dot{\mathbf{O}}^\mathrm{out}(z)$ and $\dot{\mathbf{P}}^\mathrm{out}(z)$ in different domains of the $z$-plane.}
\begin{tabular}{@{}|l|c|c|c|c|@{}}
\hline
$\mathcal{B}$ & family & $s$ & Domain & $\dot{\mathbf{P}}^\mathrm{out}(z)$ \\
\hline
\hline
\multirow{5}{*}{$\rectangle$} & \multirow{2}{*}{gO} & \multirow{2}{*}{$1$} & $\circledomain$ &
\shortstrut $\mathbf{D}(\sqrt{2}\ee^{\frac{5\ii\pi}{12}}\ee^{\frac{1}{2}\ii TR_1})\dot{\mathbf{O}}^\mathrm{out}(z)
\mathbf{D}(\sqrt{\tfrac{3}{2}}\ee^{-\frac{5\ii\pi}{12}}\ee^{-\frac{1}{2}\ii TR_1})$ \\
\cline{4-5}
&&& $\mathbb{C}\setminus\overline{\circledomain}$ & 
\shortstrut $\mathbf{D}(\sqrt{2}\ee^{\frac{5\ii\pi}{12}}\ee^{\frac{1}{2}\ii TR_1})\dot{\mathbf{O}}^\mathrm{out}(z)\mathbf{D}(\tfrac{1}{\sqrt{2}}\ee^{-\frac{5\ii\pi}{12}}\ee^{-\frac{1}{2}\ii TR_1})$ \\
\cline{2-5}
&gH & $1$ & $\mathbb{C}$ & \shortstrut $\mathbf{D}(\ee^{-\frac{\ii\pi}{2}}\ee^{\frac{1}{2}\ii TR_1})\dot{\mathbf{O}}^\mathrm{out}(z)\mathbf{D}(\ee^{\frac{\ii\pi}{2}}\ee^{-\frac{1}{2}\ii TR_1})$ \\
\cline{2-5}
&gO& \multirow{2}{*}{$-1$} &  $\circledomain$ &  \shortstrut $\mathbf{D}(\sqrt{2}\ee^{\frac{\ii\pi}{12}}\ee^{\frac{1}{2}\ii TR_1})\dot{\mathbf{O}}^\mathrm{out}(z)\mathbf{T}(\sqrt{2}\ee^{\frac{11\ii\pi}{12}}\ee^{-\frac{1}{2}\ii TR_1}\ee^{-\zeta z})$ \\
\cline{4-5}
&&& $\mathbb{C}\setminus\overline{\circledomain}$ &
\shortstrut $\mathbf{D}(\sqrt{2}\ee^{\frac{\ii\pi}{12}}\ee^{\frac{1}{2}\ii TR_1})\dot{\mathbf{O}}^\mathrm{out}(z)
\mathbf{D}(\tfrac{1}{\sqrt{2}}\ee^{-\frac{\ii\pi}{12}}\ee^{-\frac{1}{2}\ii TR_1})$ \\
\hline
\hline
\multirow{4}{*}{$\TR$} & \multirow{4}{*}{gO} & \multirow{2}{*}{$1$} & $\circledomain$ &
\shortstrut $\mathbf{D}(\sqrt{2}\ee^{\frac{\ii\pi}{4}}\ee^{-\ii\pi\Theta_\infty}\ee^{\frac{1}{2}\ii TR_2})\dot{\mathbf{O}}^\mathrm{out}(z)\mathbf{T}(\sqrt{\tfrac{2}{3}}\ee^{\frac{11\ii\pi}{12}}\ee^{-\ii T(R_1+\frac{1}{2}R_2)}\ee^{-\zeta z})$ \\
\cline{4-5}
&&& $\mathbb{C}\setminus\overline{\circledomain}$ & 
\shortstrut $\mathbf{D}(\sqrt{2}\ee^{\frac{\ii\pi}{4}}\ee^{-\ii\pi\Theta_\infty}\ee^{\frac{1}{2}\ii TR_2})\dot{\mathbf{O}}^\mathrm{out}(z)\mathbf{D}(\tfrac{1}{\sqrt{2}}\ee^{-\frac{\ii\pi}{4}}\ee^{\ii\pi\Theta_\infty}\ee^{-\frac{1}{2}\ii TR_2})$
\\
\cline{3-5}
&& \multirow{2}{*}{$-1$} &  $\circledomain$ &  \shortstrut $\mathbf{D}(\sqrt{2}\ee^{\frac{\ii\pi}{4}}\ee^{-\ii\pi\Theta_\infty}\ee^{\frac{1}{2}\ii TR_2})\dot{\mathbf{O}}^\mathrm{out}(z)\mathbf{D}(\tfrac{1}{\sqrt{2}}\ee^{-\frac{5\ii\pi}{12}}\ee^{\ii\pi\Theta_\infty}\ee^{-\ii T(R_1+\frac{1}{2}R_2)})$ \\
\cline{4-5}
&&& $\mathbb{C}\setminus\overline{\circledomain}$ &
\shortstrut $\mathbf{D}(\sqrt{2}\ee^{\frac{\ii\pi}{4}}\ee^{-\ii\pi\Theta_\infty}\ee^{\frac{1}{2}\ii TR_2})\dot{\mathbf{O}}^\mathrm{out}(z)\mathbf{D}(\tfrac{1}{\sqrt{2}}\ee^{-\frac{\ii\pi}{4}}\ee^{\ii\pi\Theta_\infty}\ee^{-\frac{1}{2}\ii TR_2})$ \\
\hline
\hline
\multirow{4}{*}{$\TI$} & \multirow{4}{*}{gO} & 
\multirow{2}{*}{$1$} & $\circledomain$ & \shortstrut $\mathbf{D}(\sqrt{2}\ee^{\frac{\ii\pi}{4}}\ee^{\ii\pi\Theta_\infty}\ee^{\frac{1}{2}\ii TR_2})\dot{\mathbf{O}}^\mathrm{out}(z)\mathbf{D}(\sqrt{\tfrac{3}{2}}\ee^{-\frac{\ii\pi}{4}}\ee^{-\ii\pi\Theta_\infty}\ee^{-\ii T(R_1+\frac{1}{2}R_2)})$ \\
\cline{4-5}
&&& $\mathbb{C}\setminus\overline{\circledomain}$ & \shortstrut $\mathbf{D}(\sqrt{2}\ee^{\frac{\ii\pi}{4}}\ee^{\ii\pi\Theta_\infty}\ee^{\frac{1}{2}\ii TR_2})\dot{\mathbf{O}}^\mathrm{out}(z)\mathbf{D}(\tfrac{1}{\sqrt{2}}\ee^{-\frac{\ii\pi}{4}}\ee^{-\ii\pi\Theta_\infty}\ee^{-\frac{1}{2}\ii TR_2})$ \\
\cline{3-5}
&&\multirow{2}{*}{$-1$} & $\circledomain$ & \shortstrut $\mathbf{D}(\sqrt{2}\ee^{\frac{\ii\pi}{4}}\ee^{\ii\pi\Theta_\infty}\ee^{\frac{1}{2}\ii TR_2})\dot{\mathbf{O}}^\mathrm{out}(z)\mathbf{T}(\sqrt{2}\ee^{\frac{3\ii\pi}{4}}\ee^{-\ii\pi\Theta_\infty}\ee^{-\ii T(R_1+\frac{1}{2}R_2)}\ee^{-\zeta z})$ \\
\cline{4-5}
&&& $\mathbb{C}\setminus\overline{\circledomain}$ & \shortstrut $\mathbf{D}(\sqrt{2}\ee^{\frac{\ii\pi}{4}}\ee^{\ii\pi\Theta_\infty}\ee^{\frac{1}{2}\ii TR_2})\dot{\mathbf{O}}^\mathrm{out}(z)\mathbf{D}(\tfrac{1}{\sqrt{2}}\ee^{-\frac{\ii\pi}{4}}\ee^{-\ii\pi\Theta_\infty}\ee^{-\frac{1}{2}\ii TR_2})$\\
\hline
\end{tabular}
\label{tab:outer-uniformize}
\end{table}

The jump contour for $\dot{\mathbf{P}}^\mathrm{out}(z)$ is a simple curve consisting of three consecutive arcs:  an arc $B_1$ from $z=\alpha$ to $z=\beta$, an arc $G$ from $z=\beta$ to $z=\gamma$, and an arc $B_2$ from $z=\gamma$ to $z=\delta$.  We think of these arcs as two ``bands'' ($B_1$ and $B_2$) separated by a ``gap'' $G$ (see the left-hand panel of Figure~\ref{fig:OuterParametrixContours}).  Defining real phases $C_\mathrm{G}$ and $C_\mathrm{B}$ as shown for each case in 
Table~\ref{tab:outer-uniformize-phases}, it is straightforward to check that $\dot{\mathbf{P}}^\mathrm{out}(z)$ is the necessarily unique solution of the following Riemann-Hilbert problem.
\begin{table}[h]
\caption{The phases $C_\mathrm{G}$ and $C_\mathrm{B}$, and the sign $\nu=\pm 1$.}
%\begin{tabular}{@{}|l|l|l|l|l|l|@{}}
%\hline
%$\mathcal{B}$ & $\kappa$ & $s$ & $C_\mathrm{G}\pmod{2\pi}$ & $C_\mathrm{B}\pmod{2\pi}$ & $\nu$\\
%\hline
%\hline
%\multirow{4}{*}{$\rectangle$} & \multirow{2}{*}{$(-1,1)$} & $1$ & $-2TR_2-\tfrac{1}{3}\pi$ &
%\shortstrut $-2TR_1+\tfrac{1}{3}\pi$ & $1$\\
%\cline{3-6}
%&& $-1$ &  $-2TR_2 +\tfrac{1}{3}\pi$ &  \shortstrut $-2TR_1-\tfrac{1}{3}\pi$ & $-1$\\
%\cline{2-6}
%&\multirow{2}{*}{$(1,+\infty)$} & $1$ & $-2TR_2-\tfrac{1}{3}\pi$ & 
%\shortstrut $-2TR_1 + \pi$ & $1$ \\
%\cline{3-6}
%&& $-1$ & $-2TR_2 +\tfrac{1}{3}\pi$ & \shortstrut $-2TR_1+\pi$ & $-1$ \\
%\hline
%\hline
%\multirow{4}{*}{$\TR$} & \multirow{2}{*}{$(-1,1)$} & $1$ & $2TR_1-\tfrac{1}{3}\pi$ &
%\shortstrut $-T(R_1+R_2)+(m-\tfrac{1}{3})\pi$ & $-1$ \\
%\cline{3-6}
%&& $-1$ &  $2TR_1+\tfrac{1}{3}\pi$ &  \shortstrut $-T(R_1+R_2)+(m+\tfrac{1}{3})\pi$ & $1$ \\
%\cline{2-6}
%&\multirow{2}{*}{$(1,+\infty)$} & $1$ & $-2TR_1+\pi$ & 
%\shortstrut $-T(R_1+R_2)+(m+1)\pi$ & $1$
%\\
%\cline{3-6}
%&& $-1$ & $-2TR_1+\pi$ & \shortstrut $-T(R_1+R_2)+(m+1)\pi$ & $-1$\\
%\hline
%\hline
%\multirow{2}{*}{$\TI$} & \multirow{2}{*}{$(-1,1)\cup(1,+\infty)$} & 
%$1$ & $2TR_1+\tfrac{1}{3}\pi$ & \shortstrut $T(R_1-R_2)+(m+\tfrac{2}{3})\pi$ & $1$\\
%\cline{3-6}
%&& $-1$ & $2TR_1-\tfrac{1}{3}\pi$ & \shortstrut $T(R_1-R_2)+(m-\tfrac{2}{3})\pi$ & $-1$ \\
%\hline
%\end{tabular}
\begin{tabular}{@{}|l|c|c|c|c|c|@{}}
\hline
$\mathcal{B}$ & family & $s$ & $C_\mathrm{G}\pmod{2\pi}$ & $C_\mathrm{B}\pmod{2\pi}$ & $\nu$\\
\hline
\hline
\multirow{3}{*}{$\rectangle$} & gO & $1$ & $-2TR_2-\tfrac{1}{3}\pi$ &
\shortstrut $-2TR_1+\tfrac{1}{3}\pi$ & $1$\\
\cline{2-6}
& gH & $1$ & \shortstrut $-2TR_2$ & $-2TR_1$ & $1$\\
\cline{2-6}
& gO & $-1$ &  $-2TR_2 +\tfrac{1}{3}\pi$ &  \shortstrut $-2TR_1-\tfrac{1}{3}\pi$ & $-1$\\
\hline
\hline
\multirow{2}{*}{$\TR$} & \multirow{2}{*}{gO} & $1$ & $2TR_1-\tfrac{1}{3}\pi$ &
\shortstrut $-T(R_1+R_2)+2\pi(\Theta_\infty+\tfrac{1}{3})$ & $-1$ \\
\cline{3-6}
&& $-1$ &  $2TR_1+\tfrac{1}{3}\pi$ &  \shortstrut $-T(R_1+R_2)+2\pi(\Theta_\infty-\tfrac{1}{3})$ & $1$ \\
\hline
\hline
\multirow{2}{*}{$\TI$} & \multirow{2}{*}{gO} & 
$1$ & $2TR_1+\tfrac{1}{3}\pi$ & \shortstrut $-T(R_1+R_2)-2\pi(\Theta_\infty+\tfrac{1}{3})$ & $1$\\
\cline{3-6}
&& $-1$ & $2TR_1-\tfrac{1}{3}\pi$ & \shortstrut $-T(R_1+R_2)-2\pi(\Theta_\infty-\tfrac{1}{3})$ & $-1$ \\
\hline
\end{tabular}
\label{tab:outer-uniformize-phases}
\end{table}

\begin{rhp}[Uniformized Outer Parametrix]
Let $\zeta\in\mathbb{C}$ and real constants $C_\mathrm{G}$ and $C_\mathrm{B}$ be given.  Seek a $2\times 2$ matrix function $z\mapsto\dot{\mathbf{P}}^\mathrm{out}(z;\zeta)$ with the following properties:
\begin{itemize}
\item\textbf{Analyticity:}  $\dot{\mathbf{P}}^\mathrm{out}(z;\zeta)$ is an analytic function of $z$ in the domain $z\in\mathbb{C}\setminus (B_1\cup G\cup B_2)$.
\item\textbf{Jump conditions:}  $\dot{\mathbf{P}}^\mathrm{out}(z;\zeta)$ assumes continuous boundary values on its jump contour from either side except at the four points $p=\alpha,\beta,\gamma,\delta$, where $(z-p)^{\frac{1}{4}}\dot{\mathbf{P}}^\mathrm{out}(z;\zeta)$ is bounded.  The boundary values are related on each arc of $B_1\cup G\cup B_2$ by the jump conditions
\eq
\dot{\mathbf{P}}^\mathrm{out}_+(z;\zeta)=\dot{\mathbf{P}}^\mathrm{out}_-(z;\zeta)\mathbf{T}(-\ee^{-\zeta z}),\quad z\in B_1,
\label{eq:Pout-jump-1}
\endeq
\eq
\dot{\mathbf{P}}^\mathrm{out}_+(z;\zeta)=\dot{\mathbf{P}}^\mathrm{out}_-(z;\zeta)\mathbf{D}(\ee^{\ii C_\mathrm{G}}),\quad z\in G,\quad\text{and}
\label{eq:Pout-jump-2}
\endeq
\eq
\dot{\mathbf{P}}^\mathrm{out}_+(z;\zeta)=\dot{\mathbf{P}}^\mathrm{out}_-(z;\zeta)\mathbf{T}(-\ee^{\ii C_\mathrm{B}}\ee^{-\zeta z}),\quad z\in B_2.
\label{eq:Pout-jump-3}
\endeq
\item\textbf{Normalization:}  $\dot{\mathbf{P}}^\mathrm{out}(z;\zeta)\to\mathbb{I}$ as $z\to\infty$.
\end{itemize}
\label{rhp:Pout}
\end{rhp}
Given existence of a solution of Riemann-Hilbert Problem~\ref{rhp:Pout}, uniqueness is straightforward to establish.  
This is an algebro-geometric problem that can be solved in terms of the function theory of the elliptic spectral curve $\mathcal{R}$ associated with the quartic polynomial $P(z)=(z-\alpha)(z-\beta)(z-\gamma)(z-\delta)$ (cf.\@ \eqref{eq:elliptic-ODE}).  To develop the solution in concrete terms, we introduce a branch of $\sqrt{P(z)}$ adapted to the jump contour at hand; 
let $r(z)$ denote the function analytic for $z\in\mathbb{C}\setminus B_1\cup B_2$ that satisfies 
$r(z)^2=P(z)$ and has asymptotic behavior $r(z)= z^2 + \bo(z)$ as $z\to\infty$.  Its domain of analyticity is the complement of the orange arcs in the right-hand panel of Figure~\ref{fig:OuterParametrixContours}.
Comparing with the orange arcs in the left-hand panels of Figures~\ref{fig:SignsContourLenses-y0-kappa0-s1}--\ref{fig:SignsContourLenses-y1p3i-kappa0-sm1} we may relate $r(z)$ with $R(z)$ explicitly:
\eq
r(z)=\begin{cases}
\nu R(z),&\quad z\in\circledomain\\
R(z),&\quad z\in\mathbb{C}\setminus\overline{\circledomain},
\end{cases}
\label{eq:r-to-R}
\endeq
where $\nu=\pm 1$ is the sign indicated for each case in Table~\ref{tab:outer-uniformize-phases}.
We then define 
\eq
%F(z) := \frac{r(z)}{2\pi \ii}\int_{B_1}\frac{-\zeta s\,\dd s}{r_+(s)(s-z)} + \frac{r(z)}{2\pi \ii}\int_{G}\frac{\ii C_\mathrm{G}\,\dd s}{r(s)(s-z)} + \frac{r(z)}{2\pi \ii}\int_{B_2}\frac{(-\zeta s+\ii C_\mathrm{B})\,\dd s}{r_+(s)(s-z)},\quad z\in\mathbb{C}\setminus (B_1\cup G\cup B_2).
F(z):=-\frac{1}{2}\zeta z+C_\mathrm{G}\frac{r(z)}{2\pi}\int_G\frac{\dd s}{r(s)(s-z)} + C_\mathrm{B}
\frac{r(z)}{2\pi}\int_{B_2}\frac{\dd s}{r_+(s)(s-z)},\quad z\in\mathbb{C}\setminus (B_1\cup G\cup B_2).
\endeq
The function $F(z)$ is analytic and bounded on its domain of definition, and its boundary values satisfy
the jump conditions 
\eq
\begin{split}
\langle F\rangle(z) & =  -\tfrac{1}{2} \zeta z, \quad  z\in B_1, \\
\Delta F(z)& = \ii C_\mathrm{G}, \quad z\in G, \\
\langle F\rangle(z) & =  -\tfrac{1}{2}\zeta z+\tfrac{1}{2}\ii C_\mathrm{B}, \quad  z\in B_2.
\end{split}
\label{eq:F-jumps}
\endeq
Also, $F(z)$ is analytic for large $|z|$ and has an expansion of the form
\eq
F(z) = F_1z + F_0 + \bo(z^{-1}),\quad z\to\infty,
\label{eq:F-expand}
\endeq
where $F_1$ and $F_0$ are independent of $z$.  We will not need the explicit form of $F_0$, but $F_1$ is given by
\eq
%F_1:= \frac{-1}{2\pi \ii}\int_{B_1}\frac{-\zeta s\,\dd s}{r_+(s)} + \frac{-1}{2\pi \ii}\int_{G}\frac{\ii C_\mathrm{G}\,\dd s}{r(s)} + \frac{-1}{2\pi \ii}\int_{B_2}\frac{(-\zeta s+\ii C_\mathrm{B})\,\dd s}{r_+(s)}.
F_1:=-\frac{1}{2}\zeta-\frac{C_\mathrm{G}}{2\pi}\int_G\frac{\dd z}{r(z)} -\frac{C_\mathrm{B}}{2\pi}\int_{B_2}\frac{\dd z}{r_+(z)}.
\label{eq:F1-1}
\endeq
Then define 
$\dot{\mathbf{Q}}^{\mathrm{out}}(z)$ via
\eq
\dot{\mathbf{Q}}^{\mathrm{out}}(z):=\ee^{F_0\sigma_3}\dot{\mathbf{P}}^{\mathrm{out}}(z)\ee^{-F(z)\sigma_3},\quad z\in\mathbb{C}\setminus (B_1\cup G\cup B_2).
\label{Qdot-from-Pdot}
\endeq
Clearly $\dot{\mathbf{Q}}^\mathrm{out}(z)$ is analytic at least for $z\in\mathbb{C}\setminus(B_1\cup G\cup B_2)$, and its boundary values are continuous except at $p=\alpha,\beta,\gamma,\delta$ where $(z-p)^{\frac{1}{4}}\dot{\mathbf{Q}}^\mathrm{out}(z)$ is bounded.  Using \eqref{eq:F-jumps} in \eqref{eq:Pout-jump-2}, and applying Morera's Theorem shows that $G$ may be removed from the jump contour, i.e., $\dot{\mathbf{Q}}^{\mathrm{out}}(z)$ is analytic for $z\in\mathbb{C}\setminus (B_1\cup B_2)$.  Using \eqref{eq:F-jumps} in \eqref{eq:Pout-jump-1} and \eqref{eq:Pout-jump-3} then shows that 
$\dot{\mathbf{Q}}^{\mathrm{out}}(z)$ satisfies jump conditions on $B_1\cup B_2$ of a universal form:
$\dot{\mathbf{Q}}^\mathrm{out}_+(z)=\dot{\mathbf{Q}}^\mathrm{out}_-(z)\mathbf{T}(-1)$, where $\mathbf{T}(-1)$ is an elementary ``twist'' matrix defined in \eqref{eq:matrix-factors-notation}.  Finally, from the normalization condition on $\dot{\mathbf{P}}^\mathrm{out}(z)$ and the expansion \eqref{eq:F-expand} one sees that $\dot{\mathbf{Q}}^\mathrm{out}(z)\ee^{F_1z\sigma_3}\to\mathbb{I}$ as $z\to\infty$.
 
These conditions on $\dot{\mathbf{Q}}^\mathrm{out}(z)$ are standard; for example, they are directly analogous to conditions defining 
the function $\mathbf{S}(\lambda)$ in \cite[Section 4]{BilmanBW:2019}.  We now summarize the construction of $\dot{\mathbf{Q}}^\mathrm{out}(z)$.  

Define a basis of homology cycles $\mathfrak{a}$ and $\mathfrak{b}$ as in the right-hand panel of Figure~\ref{fig:OuterParametrixContours}.
\begin{figure}[h]
\begin{center}
\includegraphics{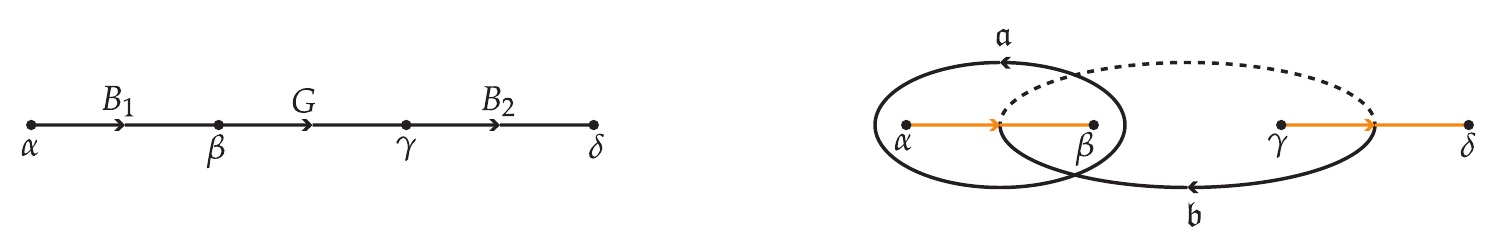}
\end{center}
\caption{Left:  topologically accurate representation of the jump contour of Riemann-Hilbert Problem~\ref{rhp:Pout} and of the Abel map $a(z)$.  Right:  corresponding jump contour for $r(z)$ and for $\dot{\mathbf{Q}}^{\mathrm{out}}(z)$, and homology cycles $\mathfrak{a}$ and $\mathfrak{b}$ on $\mathcal{R}$.}
\label{fig:OuterParametrixContours}
\end{figure}
%Here $r(z)$ is used when integrating along the thick 
%black curves while $-r(z)$ is used when integrating along the thin red 
%curves.
%\begin{figure}[h]
%\begin{center}
%\begin{picture}(100,40)(-50,-30)
%\put(-75,0){\line(1,0){50}} %Branch cuts
%\put(25,0){\line(1,0){50}}
%\thicklines
%\put(-75,0){\circle*{3}} %Branch points
%\put(-25,0){\circle*{3}}
%\put(25,0){\circle*{3}}
%\put(75,0){\circle*{3}}
%\put(-78,4){$\alpha$}
%\put(-28,4){$\beta$}
%\put(22,5){$\gamma$}
%\put(72,4){$\delta$}
%\linethickness{.02in}
%\qbezier(-85,0)(-85,20)(-50,20)  % a loop
%\qbezier(-85,0)(-85,-20)(-50,-20)
%\qbezier(-15,0)(-15,20)(-50,20)
%\qbezier(-15,0)(-15,-20)(-50,-20)
%\put(-50,20){\vector(-1,0){5}}
%\put(-53,24){$\mathfrak{a}$}
%\qbezier(-50,0)(-50,-20)(0,-20)   % b loop
%\qbezier(50,0)(50,-20)(0,-20)
%\linethickness{.005in}
%\color{red}
%\qbezier(-52.5,0)(-52.5,20)(0,20)
%\qbezier(47.5,0)(47.5,20)(0,20)
%%\multiput(-25,5)(4,0){13}{\line(1,0){2}}
%\color{black}
%\put(0,-20){\vector(-1,0){5}}
%\put(-3,-30){$\mathfrak{b}$}
%\end{picture}
%\end{center}
%\caption{The homology cycles $\mathfrak{a}$ and $\mathfrak{b}$ in relation to 
%the branch cuts of $r(z)$.  Thin red lines lie on the sheet where $r(z)$ is 
%replaced with $-r(z)$.}
%\label{homology-cycles}
%\end{figure}
The basic normalized holomorphic differential on $\mathcal{R}$ is $\omega(z)\,\dd z$, where on one sheet of $\mathcal{R}$ modeled as two copies of the complex $z$-plane cut and identified along $B_1\cup B_2$,
\eq
\omega(z):=\frac{1}{c}\cdot\frac{2\pi\ii}{r(z)}\quad\text{where}\quad c:=\oint_\mathfrak{a}\frac{\dd s}{r(s)}\quad\implies\quad \oint_\mathfrak{a}\omega(z)\,\dd z = 2\pi\ii.
\label{eq:omega}
\endeq
Note that $c$ is a concrete form of the period $Z_\mathfrak{a}$ defined in \eqref{eq:elliptic-periods}.
A meromorphic differential on $\mathcal{R}$ is $\Omega(z)\,\dd z$, where on the same sheet that \eqref{eq:omega} holds,
\eq
\label{one-form}
\Omega(z):=\frac{z^2-\tfrac{1}{2}(\alpha+\beta+\gamma+\delta)z}{r(z)}-C\omega(z),\quad
C:=\frac{1}{2\pi\ii}\oint_\mathfrak{a}\frac{z^2-\tfrac{1}{2}(\alpha+\beta+\gamma+\delta)z}{r(z)}\,\dd z\quad\implies\quad \oint_\mathfrak{a}\Omega(z)\,\dd z = 0.
\endeq
Note that $\Omega(z)\,\dd z$ has double poles at the points over $z=\infty$ on $\mathcal{R}$, but no residues.  The $\mathfrak{b}$-period of the $\mathfrak{a}$-normalized holomorphic differential $\omega(z)\,\dd z$ is
\eq \label{B-cycle}
H_\omega:=\oint_{\mathfrak{b}}\omega(z)\,\dd z=
-2\int_G\omega(z)\,\dd z,
\endeq
and that of the $\mathfrak{a}$-normalized meromorphic differential $\Omega(z)\,\dd z$ is
\eq
\label{U-def}
H_\Omega:=\oint_{\mathfrak{b}}\Omega(z)\,\dd z = -2\int_G\Omega(z)\,\dd z.
\endeq
Here the first formula in each case is a contour integral on the Riemann surface $\mathcal{R}$, and the second formula is a contour integral over the gap $G$ where the integrand as defined in \eqref{eq:omega}
and \eqref{one-form} respectively has a definite value.  A fundamental fact of the theory is that $\mathrm{Re}(H_\omega)<0$.  Note also that $H_\omega$ is a concrete version of the period ratio $2\pi\ii Z_\mathfrak{b}/Z_\mathfrak{a}$; see \eqref{eq:elliptic-periods}.
The Abel map $a(z)$ is defined by
\eq \label{Abel-map}
a(z):=\int_\alpha^z\omega(s)\,\dd s,\quad z\in\mathbb{C}\setminus(B_1\cup G\cup B_2),
\endeq
where the integral is independent of path taken in the indicated domain, on which $z\mapsto a(z)$ is holomorphic.  The analogous integral with the meromorphic differential $\Omega(z)\,\dd z$ in place of $\omega(z)\,\dd z$ is
\eq
A(z):=\int_\alpha^z\Omega(s)\,\dd s,\quad z\in\mathbb{C}\setminus(B_1\cup G\cup B_2),
\endeq
again independent of path in the indicated domain (there is no residue at $z=\infty$) and defining an analytic function on that domain.  Absence of a residue at $z=\infty$ also implies existence of the limit
\eq
\label{J-def}
J := \lim_{z \to \infty}\left(z-A(z)\right).
\endeq
It is straightforward to confirm that $a(z)$ and $A(z)$ satisfy the following jump conditions:
\eq
\begin{split}
\langle a\rangle (z) = 0\quad &\text{and}\quad \langle A\rangle (z) = 0,\quad z\in B_1,\\
\Delta a(z) = -2\pi\ii\quad &\text{and}\quad \Delta A(z) = 0,\quad z\in G,\\
\langle a\rangle (z) = -\tfrac{1}{2}H_\omega \quad &\text{and}\quad\langle A\rangle (z) = -\tfrac{1}{2}H_\Omega,\quad z\in B_2.
\label{eq:aA-jumps}
\end{split}
\endeq
Note that using \eqref{eq:omega} and \eqref{B-cycle} in \eqref{eq:F1-1}, we can rewrite $F_1$ equivalently in the form
\eq
\begin{split}
F_1&=-\frac{\zeta}{2} -\frac{C_\mathrm{G}c}{4\pi^2\ii}\int_G\omega(z)\,\dd z -\frac{C_\mathrm{B}c}{4\pi^2\ii}\int_{B_2}\omega_+(z)\,\dd z\\
&=-\frac{\zeta}{2}+\frac{C_\mathrm{G}cH_\omega}{8\pi^2\ii}-\frac{C_\mathrm{B}c}{4\pi}.
\end{split}
\label{eq:F1-2}
\endeq
A further useful identity can be found by integrating the differential $a(z)\Omega(z)\,\dd z$ around the boundary of the canonical dissection of $\mathcal{R}$ and expressing the integral alternately via residues at the two poles and via periods occurring on the four edges of the boundary.  The result is that $H_\Omega$ defined in \eqref{U-def} can be expressed in terms of $c$ defined in \eqref{eq:omega} as
\eq
H_\Omega=\frac{4\pi\ii}{c}
\label{eq:U-alt}
\endeq
See \cite{Dubrovin81} or \cite[Lemma B.1]{BuckinghamM13} for details.

The analogue of the function $j(z)$ from Section~\ref{sec:Genus-Zero-Parametrix} is here defined as the unique function analytic for $z\in\mathbb{C}\setminus (B_1\cup B_2)$ satisfying the conditions
\eq
j(z)^4 := \frac{(z-\alpha)(z-\gamma)}{(z-\beta)(z-\delta)}\quad\text{and}\quad \lim_{z\to\infty}j(z)=1.
\endeq
This function satisfies the jump condition 
$j_+(z)=-\ii j_-(z)$ for $z\in B_1\cup B_2$.  Further define
\eq
f^\mathrm{D}(z) := \frac{j(z)+j(z)^{-1}}{2},\quad f^\mathrm{OD}(z) := \frac{j(z)-j(z)^{-1}}{2\ii}
\label{eq:fD-fOD-define}
\endeq
with jump conditions 
$f_+^\mathrm{D}(z)=f_-^\mathrm{OD}(z)$ and $f_+^\mathrm{OD}(z)=-f_-^\mathrm{D}(z)$ for 
$z\in B_1\cup B_2$ and large-$z$ asymptotic expansions $f^\mathrm{D}(z)=1+\bo(z^{-1})$ and $f^\mathrm{OD}(z)=\tfrac{1}{4}\ii(\alpha-\beta+\gamma-\delta)z^{-1}+\bo(z^{-2})$.  Clearly the product $r(z)f^\mathrm{D}(z)f^\mathrm{OD}(z)$ is an entire function with $r(z)f^\mathrm{D}(z)f^\mathrm{OD}(z)=\bo(z)$ as $z\to\infty$ and hence is a linear function. 
Let $z=z_0$ denote the unique root of this function (possibly $z_0=\infty$ if and only if $f^\mathrm{OD}(z)=\bo(z^{-2})$ as $z\to\infty$).  Explicitly,
\eq
r(z)f^\mathrm{D}(z)f^\mathrm{OD}(z)=-\frac{1}{4\ii}\left([\alpha-\beta+\gamma-\delta]z-[\alpha\gamma-\beta\delta]\right)\quad\implies\quad
z_0:=\frac{\alpha\gamma-\beta\delta}{\alpha-\beta+\gamma-\delta}.
\label{eq:z0-define}
\endeq
A direct computation shows that $z_0$ cannot coincide with any of the roots $z=\alpha,\beta,\gamma,\delta$ provided the latter are distinct.  Therefore the simple root $z_0$ belongs to exactly one of the two factors $f^\mathrm{D}(z)$ or $f^\mathrm{OD}(z)$. 
We assume in what follows that $f^\mathrm{OD}(z_0)=0$; 
the modifications necessary to handle the other case $f^\mathrm{D}(z_0)=0$ are 
explained in \cite[Section 4.4.2]{BothnerM:2018}.  

Finally, to construct $\dot{\mathbf{Q}}^\mathrm{out}(z)$ and hence also $\dot{\mathbf{P}}^\mathrm{out}(z)$, we introduce
the Riemann theta function of the elliptic curve $\mathcal{R}$ for the homology basis $(\mathfrak{a},\mathfrak{b})$ defined by 
\eq 
\Theta(z) :=\sum_{k\in \mathbb{Z}}\ee^{\frac{1}{2}H_\omega k^2}\ee^{kz}.
\label{theta-function}
\endeq
This is an entire function of $z$ that satisfies 
\eq
\Theta(-z)=\Theta(z),\quad \Theta(z+2\pi \ii)=\Theta(z),\quad \Theta(z+H_\omega)=\ee^{-\frac{1}{2}H_\omega}\ee^{-z}\Theta(z)
\label{eq:theta-identities}
\endeq
and has simple zeros only at the points of a $\mathbb{Z}^2$-lattice:
\eq
\Theta(z)=0 \quad \text{if and only if} \quad z= K + 2\pi \ii k + H_\omega \ell \quad \text{for} \quad k,\ell\in\mathbb{Z},
\label{eq:Theta-divisor}
\endeq
where $K=K(H_\omega):=\ii\pi +\tfrac{1}{2}H_\omega$ is one of the zeros.
See \cite[Chapter 20]{DLMF}, which uses Jacobi's notation $\Theta(z;H_\omega)=\theta_3(w|\tau)=\theta_3(w,q)$ where $z=2\ii w$, $H_\omega=2\pi\ii\tau$, and $q=\ee^{\ii\pi\tau}$.

The final result is that the matrix elements of  
$\dot{\mathbf{Q}}^{\mathrm{out}}(z)$ are
\eq 
\begin{split}
\dot{Q}^\mathrm{out}_{11}(z) & := f^\text{D}(z)\frac{\Theta(a(\infty)+a(z_0)+K)\Theta(a(z)+a(z_0)+K-F_1H_\Omega)}{\Theta(a(\infty)+a(z_0)+K-F_1H_\Omega)\Theta(a(z)+a(z_0)+K)} \ee^{-F_1[J+A(z)]}, \\
\dot{Q}^\mathrm{out}_{12}(z) & := -f^\text{OD}(z)\frac{\Theta(a(\infty)+a(z_0)+K)\Theta(a(z)-a(z_0)-K+F_1H_\Omega)}{\Theta(a(\infty)+a(z_0)+K-F_1H_\Omega)\Theta(a(z)-a(z_0)-K)} \ee^{-F_1[J-A(z)]}, \\
\dot{Q}^\mathrm{out}_{21}(z) & := f^\text{OD}(z)\frac{\Theta(a(\infty)+a(z_0)+K)\Theta(a(z)-a(z_0)-K-F_1H_\Omega)}{\Theta(a(\infty)+a(z_0)+K+F_1H_\Omega)\Theta(a(z)-a(z_0)-K)} \ee^{F_1[J-A(z)]},\\
\dot{Q}^\mathrm{out}_{22}(z) & := f^\text{D}(z)\frac{\Theta(a(\infty)+a(z_0)+K)\Theta(a(z)+a(z_0)+K+F_1H_\Omega)}{\Theta(a(\infty)+a(z_0)+K+F_1H_\Omega)\Theta(a(z)+a(z_0)+K)} \ee^{F_1[J+A(z)]}.
\end{split}
\label{Qdot-entries}
\endeq
From \eqref{Qdot-from-Pdot} and 
\eqref{Qdot-entries}, we then obtain a formula for the solution $\dot{\mathbf{P}}^\mathrm{out}(z)$ of Riemann-Hilbert Problem~\ref{rhp:Pout}.  Note that the quantity $F_1H_\Omega$ appearing in the arguments of the theta functions can be expressed via \eqref{eq:F1-2}--\eqref{eq:U-alt} as
\eq
F_1H_\Omega=\ii\varphi,\quad \varphi:=-\frac{2\pi\zeta}{c}-\xi,\quad\xi:=C_\mathrm{B}-C_\mathrm{G}\frac{H_\omega}{2\pi\ii}.
\label{eq:F1-U}
\endeq
The assumption that $f^\mathrm{D}(z)$ is nonvanishing while $f^\mathrm{OD}(z)=0$ only vanishes at $z=z_0$ to first order implies that $\Theta(a(z)+a(z_0)+K)$ is nonvanishing (even in the limit $z\to\infty$) and that the unique simple zero of $\Theta(a(z)-a(z_0)-K)$, located at $z=z_0$ by Riemann's Theorem, is cancelled by that of $f^\mathrm{OD}(z)$.
Therefore, it is obvious that $\dot{\mathbf{P}}^\mathrm{out}(z)$ exists if and only if
\eq
\Theta(a(\infty)+a(z_0)+K-\ii\varphi)\Theta(a(\infty)+a(z_0)+K+\ii\varphi)\neq 0.
\label{eq:Outer-solvable}
\endeq
\begin{lemma}
$\Theta(a(\infty)+a(z_0)+K-\ii\varphi)=0$ if and only if $\Theta(a(\infty)+a(z_0)+K+\ii\varphi)=0$.
\label{lem:SameDivisor}
\end{lemma}
\begin{proof}
Since $\Theta(-z)=\Theta(z)$, $\Theta(K)=0$, and $2K$ is a (quasi-)period, it is sufficient to prove that there are integers $N_\mathfrak{a}$ and $N_\mathfrak{b}$ such that $2a(\infty)+2a(z_0)=2\pi\ii N_\mathfrak{a} + H_\omega N_\mathfrak{b}$.  Introducing a second auxiliary copy of the complex $z$-plane on which the Abel mapping $a(z)$ is defined with the opposite sign compared to the principal $z$-plane, we obtain a two-sheeted model for the Riemann surface $\mathcal{R}$.  Denoting by $Q^+(z)$ (resp., $Q^-(z)$) the point $Q\in\mathcal{R}$ over $z\in\mathbb{C}$ on the principal (resp., auxiliary) sheet, we have therefore extended the Abel mapping to $\mathcal{R}$ by the definition $\widetilde{a}(Q^\pm(z))=\pm a(z)$.  With this notation, we want to show that $\widetilde{a}(Q^+(\infty))+\widetilde{a}(Q^+(z_0))-\widetilde{a}(Q^-(\infty))-\widetilde{a}(Q^-(z_0))=2\pi\ii N_\mathfrak{a}+H_\omega N_\mathfrak{b}$.  But by the Abel-Jacobi Theorem, this identity will hold for some integers $N_\mathfrak{a}$ and $N_\mathfrak{b}$ if there exists a nonzero meromorphic function $k(Q)$ defined on $\mathcal{R}$ with simple poles at the points $Q^-(\infty)$ and $Q^-(z_0)$ only and vanishing at the points $Q^+(\infty)$ and $Q^+(z_0)$.  Similarly extending $R(z)$ to the Riemann surface $\mathcal{R}$ by $\widetilde{R}(Q^\pm(z))=\pm R(z)$, we can exhibit this function $k(Q)$ explicitly in the form
\eq
k(Q)=\frac{[z(Q)^2-z_0^2]-\tfrac{1}{2}(\alpha+\beta+\gamma+\delta)[z(Q)-z_0]-[\widetilde{R}(Q)-\widetilde{R}(Q^+(z_0))]}{z(Q)-z_0},
\endeq
where $z(Q)$ is the coordinate (sheet projection) function satisfying $z(Q^\pm(z))=z$ for all $z\in\mathbb{C}$.  It is easy to verify from this formula that $k(Q)$ has simple poles only at the points $Q=Q^-(\infty)$ and $Q=Q^-(z_0)$, and that $k(Q)$ has a simple zero at $Q=Q^+(\infty)$.  The fact that $k(Q^+(z_0))=0$ as well amounts to the condition that the derivative of the numerator vanishes at $Q=Q^+(z_0)$, which reads
\eq
2z_0-\tfrac{1}{2}(\alpha+\beta+\gamma+\delta)=R'(z_0).
\endeq
The \emph{squares} of both sides of this equation are equal, as can be verified easily using the definition \eqref{eq:z0-define} of $z_0$.  That the equation itself holds can be proved by verifying it (i.e., checking the signs of both sides) for just one suitably chosen point $y_0$ in each domain $\mathcal{B}=\rectangle(\kappa)$ or $\TR(\kappa)$ or $\TI(\kappa)$ and invoking a continuation argument (applying under the same conditions that $z_0$ is a root of $f^\mathrm{OD}(z)$ and not of $f^\mathrm{D}(z)$, and requiring similar modifications to the proof as indicated should these conditions fail).  For this purpose, it is  convenient to consider $y_0$ on the real (resp., imaginary) axis very close to the boundary between $\rectangle(\kappa)$ and $\TR(\kappa)$ (resp., $\TI(\kappa)$), and applying a perturbation argument to open the pair of branch points that merge as $y_0$ approaches the boundary.
For a completely analogous result, see \cite[Lemma 1]{BothnerM:2018}.  Note that it is not possible that both of the integers $N_\mathfrak{a}$ and $N_\mathfrak{b}$ are even, as that would mean that $a(\infty)+a(z_0)$ is an integer linear combination of periods, which by Abel-Jacobi would imply the existence of a nonzero meromorphic function with just one simple pole at $Q=Q^-(z_0)$ that vanishes at $Q=Q^+(\infty)$.  But having at most one simple pole on $\mathcal{R}$, the function is a constant, and vanishing at any point forces it to vanish identically.
\end{proof}
In light of this result, Riemann-Hilbert Problem~\ref{rhp:Pout} is solvable if $d>0$ holds, where (considering the first factor in \eqref{eq:Outer-solvable} and using \eqref{eq:Theta-divisor} and \eqref{eq:F1-U})
% and insist that $d>0$ holds, where
%or equivalently using \eqref{eq:Theta-divisor} and \eqref{eq:F1-U}, $d^+>0$ and $d^->0$ both hold, where
\eq
%d^\pm:=\inf_{(k,\ell)\in\mathbb{Z}^2}\left|a(\infty)+a(z_0)\pm\frac{2\pi\ii\zeta}{c}+\left[k\pm\frac{C_\mathrm{B}}{2\pi}\right]2\pi\ii + 
%\left[\ell\mp\frac{C_\mathrm{G}}{2\pi}\right]H_\omega\right|.
d:=\inf_{(N_\mathfrak{a},N_\mathfrak{b})\in\mathbb{Z}^2}\left|a(\infty)+a(z_0)+\frac{2\pi\ii\zeta}{c}+\left[N_\mathfrak{a}+\frac{C_\mathrm{B}}{2\pi}\right]2\pi\ii + 
\left[N_\mathfrak{b}-\frac{C_\mathrm{G}}{2\pi}\right]H_\omega\right|.
\label{eq:non-Malgrange}
\endeq
The quantity $d$ measures the distance to the (common, by Lemma~\ref{lem:SameDivisor}) theta divisor for the two factors in \eqref{eq:Outer-solvable}.  
It depends on the data (distinct roots $\alpha,\beta,\gamma,\delta$ of the quartic $P(z)$, complex constant $\zeta$, and real constants $C_\mathrm{G}$ and $C_\mathrm{B}$ modulo $2\pi$) in the formulation of Riemann-Hilbert Problem~\ref{rhp:Pout}.  For given data, the infimum in \eqref{eq:non-Malgrange} is clearly attained, as the absolute value grows with $(N_\mathfrak{a},N_\mathfrak{b})$ because $\mathrm{Re}(H_\omega)<0$, so it is really a minimum over finitely many lattice points.    
%Given $\epsilon\ge 0$, we define conditions $C^\pm(\epsilon)$ on the data as follows:
%\eq
%\text{Condition $C^\pm(\epsilon)$ holds means:}\quad \inf_{(\ell,m)\in\mathbb{Z}^2} |a(\infty)+a(z_0)\pm F_1H_\Omega-2\pi\ii\ell-H_\omega m|>\epsilon.
%\label{eq:Cpm-epsilon}
%\endeq
Recalling the interpretation of the parameters of Riemann-Hilbert Problem~\ref{rhp:Pout} in terms of the outer parametrix $\dot{\mathbf{O}}^\mathrm{out}(z)$, we see that for given $\kappa\in(-1,1)$ and $s=\pm 1$, $d$ is a function of $y_0$ in the Boutroux domain $\mathcal{B}$ of interest (which determines $a(\infty)+a(z_0)$, $c$, $H_\omega$, and the constants $R_1$ and $R_2$ appearing in $C_\mathrm{G}$ and $C_\mathrm{B}$, via the roots of $P(z)$), $\zeta\in\mathbb{C}$, and $T>0$.  In this context, given $\epsilon>0$ we set
\eq
%\mathcal{S}(\epsilon):=\left\{ y_0\in \mathcal{B}, \zeta\in\mathbb{C}, T>0:  d^+\ge \epsilon\;\text{and}\; d^-\ge\epsilon\right\}.
\mathcal{S}(\epsilon):=\left\{ y_0\in \mathcal{B}, \zeta\in\mathbb{C}, T>0:  d\ge \epsilon\right\}.
\label{eq:Malgrange-cheese}
\endeq
Since $C_\mathrm{G}$ and $C_\mathrm{B}$ are affine linear in $T$, which does not otherwise appear in the data for Riemann-Hilbert Problem~\ref{rhp:Pout}, the set $\mathcal{S}(\epsilon)$ contains arbitrarily large values of $T$.
Expressing $\dot{\mathbf{O}}^\mathrm{out}(z)$ in terms of $\dot{\mathbf{P}}^\mathrm{out}(z)$ using Table~\ref{tab:outer-uniformize} we have the following result.
\begin{proposition}
The outer parametrix $\dot{\mathbf{O}}^\mathrm{out}(z)$ exists with $\det(\dot{\mathbf{O}}^\mathrm{out}(z))=1$ if $d>0$ holds, in which case for every $\mu>0$,
\eq
\mathcal{M}(\mu):=\mathop{\sup_{|z-p|>\mu}}_{p=\alpha,\beta,\gamma,\delta}\|\dot{\mathbf{O}}^\mathrm{out}(z)\|
\endeq
is finite, where $\|\cdot\|$ denotes any matrix norm.
If $\zeta$ is bounded, $y_0$ lies in a compact subset of the relevant Boutroux domain, and for some $\epsilon>0$, $(y_0,\zeta,T)\in \mathcal{S}(\epsilon)$, then for every $\mu>0$,
$\mathcal{M}(\mu)$ is uniformly bounded even as $T\to+\infty$.  
\label{prop:G1-outer}
\end{proposition}
\begin{proof}
It suffices to replace $\dot{\mathbf{O}}^\mathrm{out}(z)$ with $\dot{\mathbf{P}}^\mathrm{out}(z)$, because the transformation relating them given in Table~\ref{tab:outer-uniformize} always exists, preserves determinants, and is uniformly bounded.
It is easy to check from the conditions of Riemann-Hilbert Problem~\ref{rhp:Pout} that if $\dot{\mathbf{P}}^\mathrm{out}(z)$ exists, it must have unit determinant, and the fact that existence is guaranteed by the condition $d>0$ has already been proven above.  That $\mathcal{M}(\mu)<\infty$ for $\mu>0$ then follows by the maximum modulus principle.  Existence of a uniform bound of $\mathcal{M}(\mu)$ depending only on the compact set containing $y_0$, the bound for $\zeta$, and the value of $\epsilon>0$ is not obvious from the explicit construction of the solution, because the arguments of the theta functions have real parts proportional to the large parameter $T$.  However the existence of such a bound can be seen easily from the conditions of Riemann-Hilbert Problem~\ref{rhp:Pout}, in which the dependence on $T$ entering through the real phases $C_\mathrm{G}$ and $C_\mathrm{B}$ linear in $T$ is controlled because only the bounded exponentials $\ee^{\ii C_\mathrm{G}}$ and $\ee^{\ii C_\mathrm{B}}$ appear in the problem.
\end{proof}

We conclude this discussion of the outer parametrix by giving some formulas for the quantities extracted from $\dot{\mathbf{O}}^\mathrm{out}(z)$ needed to write approximate formul\ae\ for $u(x)$ and $u_\tw(x)$ in Section~\ref{sec:G1-asymptotic-formulae} below.
Expanding for large $z$ we obtain the convergent Laurent expansion
\eq
\dot{\mathbf{O}}^\mathrm{out}(z)=\mathbb{I}+\sum_{k=1}^\infty z^{-k}\dot{\mathbf{O}}^\mathrm{out}_k,
\label{eq:OdotOut-elliptic-z-large}
\endeq
in which
\eq
\dot{O}^\mathrm{out}_{1,12}=\eta \dot{P}^\mathrm{out}_{1,12},\quad \dot{P}^\mathrm{out}_{1,12}:=\lim_{z\to\infty}z\dot{P}^\mathrm{out}_{12}(z)
\label{eq:OP-infty-12}
\endeq
where $\eta$ is determined from the diagonal $z$-independent conjugation relating $\dot{\mathbf{O}}^\mathrm{out}(z)$ and $\dot{\mathbf{P}}^\mathrm{out}(z)$ for large $z$ given in Table~\ref{tab:outer-uniformize}:
\eq
\eta:=\begin{cases}
\tfrac{1}{2}\ee^{-\frac{5\ii\pi}{6}}\ee^{-\ii TR_1},&\quad\text{$y_0\in\rectangle(\kappa)$, gO case, $s=1$}\\
-\ee^{-\ii TR_1},&\quad\text{$y_0\in\rectangle(\kappa)$, gH case, $s=1$}\\
\tfrac{1}{2}\ee^{-\frac{\ii\pi}{6}}\ee^{-\ii TR_1},&\quad\text{$y_0\in\rectangle(\kappa)$, gO case, $s=-1$}\\
\tfrac{1}{2}\ee^{-\frac{\ii\pi}{2}}\ee^{2\pi\ii\Theta_\infty}\ee^{-\ii TR_2},&\quad
\text{$y_0\in\TR(\kappa)$, gO case, $s=\pm 1$}\\
\tfrac{1}{2}\ee^{-\frac{\ii\pi}{2}}\ee^{-2\pi\ii\Theta_\infty}\ee^{-\ii TR_2},&\quad
\text{$y_0\in\TI(\kappa)$, gO case, $s=\pm 1$}.
\end{cases}
\endeq
Also, evaluation at $z=0$ gives 
\eq
\dot{O}^\mathrm{out}_{11}(0)\dot{O}^\mathrm{out}_{12}(0)=\nu\eta\dot{P}^\mathrm{out}_{11}(0)\dot{P}^\mathrm{out}_{12}(0)
\label{eq:OP-zero-product}
\endeq
and
\eq
\frac{\dot{O}^\mathrm{out}_{21}(0)}{\dot{O}^\mathrm{out}_{11}(0)}=\frac{1}{\eta}\begin{cases}
\displaystyle
\frac{\dot{P}^\mathrm{out}_{21}(0)}{\dot{P}^\mathrm{out}_{11}(0)},&\quad\nu=1\\
\displaystyle
\frac{\dot{P}^\mathrm{out}_{22}(0)}{\dot{P}^\mathrm{out}_{12}(0)},&\quad\nu=-1,
\end{cases}
\endeq
where the sign $\nu=\pm 1$ is as given in Table~\ref{tab:outer-uniformize-phases}.
Using the explicit solution of Riemann-Hilbert Problem~\ref{rhp:Pout}, we have 
\eq
\dot{P}^\mathrm{out}_{1,12}=\mathcal{N}(y_0)\frac{\Theta(a(\infty)-a(z_0)-K+\ii\varphi)}{\Theta(a(\infty)+a(z_0)+K-\ii\varphi)} 
\ee^{-2F_1J}
\label{eq:P-infty-12}
\endeq
and
\eq
\begin{split}
\dot{P}^\mathrm{out}_{11}(0) & = \mathcal{N}_{11}(y_0)\frac{\Theta(a(0)+a(z_0)+K-\ii\varphi)}{\Theta(a(\infty)+a(z_0)+K-\ii\varphi)} \ee^{-F_1[J+A(0)]}, \\
\dot{P}^\mathrm{out}_{12}(0) & = \mathcal{N}_{12}(y_0)\frac{\Theta(a(0)-a(z_0)-K+\ii\varphi)}{\Theta(a(\infty)+a(z_0)+K-\ii\varphi)} \ee^{-F_1[J-A(0)]}, \\
\dot{P}^\mathrm{out}_{21}(0) & = \mathcal{N}_{21}(y_0)\frac{\Theta(a(0)-a(z_0)-K-\ii\varphi)}{\Theta(a(\infty)+a(z_0)+K+\ii\varphi)}\ee^{F_1[J-A(0)]},\\
\dot{P}^\mathrm{out}_{22}(0) & = \mathcal{N}_{22}(y_0)\frac{\Theta(a(0)+a(z_0)+K+\ii\varphi)}{\Theta(a(\infty)+a(z_0)+K+\ii\varphi)}
\ee^{F_1[J+A(0)]},
\end{split}
\label{eq:P-zero}
\endeq
in which
\eq
\mathcal{N}(y_0):=\frac{\alpha-\beta+\gamma-\delta}{4\ii}\frac{\Theta(a(\infty)+a(z_0)+K)}{\Theta(a(\infty)-a(z_0)-K)}\ee^{-2F_0},
\label{eq:calN-define}
\endeq
and
\eq
\begin{split}
\mathcal{N}_{11}(y_0)&:=f^\mathrm{D}(0)\frac{\Theta(a(\infty)+a(z_0)+K)}{\Theta(a(0)+a(z_0)+K)}\ee^{F(0)-F_0},\\
\mathcal{N}_{12}(y_0)&:=-f^\mathrm{OD}(0)\frac{\Theta(a(\infty)+a(z_0)+K)}{\Theta(a(0)-a(z_0)-K)}\ee^{-F(0)-F_0},\\
\mathcal{N}_{21}(y_0)&:=f^\mathrm{OD}(0)\frac{\Theta(a(\infty)+a(z_0)+K)}{\Theta(a(0)-a(z_0)-K)}\ee^{F(0)+F_0},\\
\mathcal{N}_{22}(y_0)&:=f^\mathrm{D}(0)\frac{\Theta(a(\infty)+a(z_0)+K)}{\Theta(a(0)+a(z_0)+K)}\ee^{-F(0)+F_0},
\end{split}
\label{eq:calN-jk-define}
\endeq
are all finite and nonzero (all apparent singularities on the parameter space are removable).

\subsubsection{Inner parametrices}
The approximation of the jump matrices for $\mathbf{O}(z)$ by those of its outer parametrix $\dot{\mathbf{O}}^\mathrm{out}(z)$ fails to be uniformly accurate when $z$ is near the four points $z=\alpha,\beta,\gamma,\delta$.  To deal with this nonuniformity and also to avoid the problematic divergence of $\dot{\mathbf{O}}^\mathrm{out}(z)$ at these four points, we define instead four \emph{inner parametrices}.

Let $D_p$, $p\in\{\alpha,\beta,\gamma,\delta\}$ be fixed small disks in the $z$-plane with center $z=p$.  Within each disk we define a conformal map $W:D_p\to\mathbb{C}$ such that $W(p)=0$, as indicated in the Tables~\ref{tab:Inner-y0-kappa0-s1}--\ref{tab:Inner-y1p3i-kappa0-sm1} in Appendix~\ref{app:DiagramsAndTables}.  It is assumed that certain contour arcs are first ``fused'' as indicated in these tables with the ``\&'' notation, and the fact that $W(z)$ is conformal follows from the explicit formula $h'(z)^2=P(z)/(16z^2)$ and the fact that $z=p$ is in each case a simple root of the polynomial $P(z)$.   
Then, for $z\in D_p$, we define a matrix $\mathbf{P}^p(z)$ by the formula
\eq
\mathbf{P}^p(z):=\mathbf{O}(z)\ee^{(-Th(z)+\tfrac{1}{2}\zeta z)\sigma_3}\mathbf{C}(z)\ee^{-\tfrac{1}{2}TW(z)^{3/2}\sigma_3},\quad z\in D_p,
\label{eq:Airy-O-to-P}
\endeq
where $\mathbf{C}(z)$ is a different piecewise-constant matrix function on each disk as indicated in the same series of tables.  It is easy to check that $\mathbf{P}^p(z)$ satisfies exactly the same jump and analyticity conditions as does the matrix $\mathbf{A}(T^{2/3}W(z))$, where $\mathbf{A}(\xi)$ is defined in Section~\ref{sec:Genus-Zero-Parametrix}.  If we replace $\mathbf{O}(z)$ with $\dot{\mathbf{O}}^\mathrm{out}(z)$ in \eqref{eq:Airy-O-to-P}, then we get instead a matrix function analytic in $D_p$ except where $\xi=T^{2/3}W(z)<0$, across which arc the jump condition \eqref{eq:Airy-jump-last} is satisfied.  It follows that the function 
\eq
\mathbf{H}^p(z):=\dot{\mathbf{O}}^\mathrm{out}(z)\ee^{(-Th(z)+\tfrac{1}{2}\zeta z)\sigma_3}\mathbf{C}(z)\ee^{-\tfrac{1}{2}TW(z)^{3/2}\sigma_3}\frac{1}{\sqrt{2}}\bpm 1 & -1\\1 & 1\epm W(z)^{-\sigma_3/4},\quad z\in D_p
\label{eq:HOut-def}
\endeq
can be extended to $W(z)<0$ from both sides so as to become an analytic function on $D_p$ (any apparent singularity at $z=p$ is easily seen to be removable due to the allowed growth condition on $\dot{\mathbf{O}}^\mathrm{out}(z)$ as $z\to p$).  From Tables~\ref{tab:Inner-y0-kappa0-s1}--\ref{tab:Inner-y1p3i-kappa0-sm1} one can check that the matrix product $\ee^{-Th(z)\sigma_3}\mathbf{C}(z)\ee^{-\tfrac{1}{2}TW(z)^{3/2}\sigma_3}$ is independent of $z$ on each subdomain of $D_p$ on which $\mathbf{C}(z)$ itself is constant, and that this matrix product is oscillatory and uniformly bounded in $T$ since $p\in\{\alpha,\beta,\gamma,\delta\}$ lies on the Stokes graph so $\mathrm{Re}(h(p))=0$ unambiguously.  Therefore, 
if $\zeta\in\mathbb{C}$ is bounded, $y_0$ lies in a given compact subset of the relevant Boutroux domain, and the parameters satisfy $(y_0,\zeta,T)\in\mathcal{S}(\epsilon)$ for some $\epsilon>0$ (see \eqref{eq:Malgrange-cheese}), then it follows from Proposition~\ref{prop:G1-outer} that
%if $\zeta\in\mathbb{C}$ remains bounded as $T\to\infty$ and moreover varies in a certain set depending on $y_0,\kappa,s,T$ on which 
%$\sup_{z\in\partial D_p}\|\dot{\mathbf{O}}^\mathrm{out}(z)\|$ is bounded \textcolor{red}{(to be clarified later)}, 
$\mathbf{H}^p(z)$ is uniformly bounded on $D_p$ independently of $T>0$.  We define an inner parametrix for $\mathbf{O}(z)$ on $D_p$ by setting
\eq
\dot{\mathbf{O}}^{\mathrm{in},p}(z):=\mathbf{H}^p(z)T^{-\sigma_3/6}\mathbf{A}(T^{2/3}W(z))\ee^{\frac{1}{2}TW(z)^{3/2}\sigma_3}\mathbf{C}(z)^{-1}\ee^{(Th(z)-\frac{1}{2}\zeta z)\sigma_3},\quad z\in D_p.
\endeq
This function satisfies exactly the same analyticity and jump conditions within $D_p$ as does $\mathbf{O}(z)$ itself.  Moreover, by using \eqref{eq:HOut-def} to express $\dot{\mathbf{O}}^\mathrm{out}(z)$ in terms of $\mathbf{H}^p(z)$, from \eqref{eq:Airy-Normalize} we see that 
\eq
\begin{split}
\dot{\mathbf{O}}^{\mathrm{in},p}(z)\dot{\mathbf{O}}^\mathrm{out}(z)^{-1}&=\mathbf{H}^p(z)T^{-\sigma_3/6}
\mathbf{A}(T^{2/3}W(z))\frac{1}{\sqrt{2}}\bpm 1&-1\\1&1\epm (TW(z))^{-\sigma_3/4}T^{\sigma_3/6}\mathbf{H}^p(z)^{-1}\\ &=\mathbf{H}^p(z)\bpm 1+\bo(T^{-2}) & \bo(T^{-1})\\\bo(T^{-1}) & 1+\bo(T^{-2})\epm\mathbf{H}^p(z)^{-1} \\ &= \mathbb{I}+\bo(T^{-1}),\quad z\in\partial D_p,
\end{split}
\label{eq:Genus-One-Airy-Match}
\endeq
since $W(z)$ is independent of $T$ and is bounded away from zero on $\partial D_p$.

\subsubsection{Global parametrix and error estimation}
The \emph{global parametrix} for $\mathbf{O}(z)$ is defined in terms of the outer and inner parametrices as
\eq
\dot{\mathbf{O}}(z):=\begin{cases}\dot{\mathbf{O}}^{\mathrm{in},p}(z),&\quad z\in D_p,\quad p\in \{\alpha,\beta,\gamma,\delta\},\\
\dot{\mathbf{O}}^\mathrm{out}(z),&\quad\text{elsewhere that $\dot{\mathbf{O}}^\mathrm{out}(z)$ is analytic},
\end{cases}
\endeq
which is just the direct analogue of \eqref{eq:GenusZero-global} in the case of four simple roots of the quartic $P(z)$.  The error matrix $\mathbf{E}(z)$ is then defined in terms of 
$\mathbf{O}(z)$ and $\dot{\mathbf{O}}(z)$ by \eqref{eq:GenusZeroE}, exactly as in the simpler setting considered in Section~\ref{sec:GenusZero-error}.  The analysis of the error described in Section~\ref{sec:GenusZero-error} applies nearly verbatim, with only the additional hypothesis that the parameters admit a uniform bound for $\dot{\mathbf{O}}^\mathrm{out}(z)$ when $z\in\mathbb{C}\setminus (D_\alpha\cup D_\beta\cup D_\gamma\cup D_\delta)$.  However, according to Proposition~\ref{prop:G1-outer}, such a bound is guaranteed by a uniform upper bound on $|\zeta|$, confinement of $y_0$ to a fixed compact subset of the Boutroux domain of interest, and the condition $(y_0,\zeta,T)\in\mathcal{S}(\epsilon)$ for some $\epsilon>0$.  Therefore, under these conditions the expansion \eqref{eq:GenusZero-E-expansion} holds with $\mathbf{E}_1=\bo(T^{-1})$ and $\mathbf{E}(0)=\mathbb{I}+\bo(T^{-1})$.
%\eq
%\mathbf{E}(z):=\mathbf{O}(z)\dot{\mathbf{O}}(z)^{-1}.
%\label{eq:Genus-One-Error}
%\endeq
%Under the assumption on $\zeta$ relative to $y_0,\kappa,s,T$ such that $\sup_{z\not\in D_\alpha\cup D_\beta\cup D_\gamma\cup D_\delta}\|\dot{\mathbf{O}}^\mathrm{out}(z)\|\le C$ for some constant $C$, it is easy to check that $\mathbf{E}(z)$ satisfies the conditions of an identity-normalized Riemann-Hilbert problem with a jump matrix $\mathbf{V}_\mathbf{E}$ for which $\mathbf{V}_\mathbf{E}-\mathbb{I}\in L^2\cap L^\infty$ with $\|\mathbf{V}_\mathbf{E}-\mathbb{I}\|_\infty = O(T^{-1})$, with the dominant contributions coming from the circles $\partial D_p$, $p\in\{\alpha,\beta,\gamma,\delta\}$, as follows from \eqref{eq:Genus-One-Airy-Match}.  Hence $\mathbf{E}(z)$ satisfies a Riemann-Hilbert problem of small-norm type when $T>0$ is large.  

\subsection{Asymptotic formul\ae\ for the rational solutions of Painlev\'e-IV on Boutroux domains}
\label{sec:G1-asymptotic-formulae}
%Recall the error matrix $\mathbf{E}(z)$ defined by \eqref{eq:Genus-One-Error}.  Small-norm theory implies that if $\zeta$ is bounded and $y_0$ is in a compact subset of $\rectangle$, $\TR$, or $\TI$, \textcolor{red}{(plus auxiliary conditions to avoid the Malgrange divisor of the outer parametrix)} $\mathbf{E}(z)$ exists for $T>0$ sufficiently large and satisfies $\mathbf{E}(0)=\mathbb{I}+O(T^{-1})$ and $\mathbf{E}(z)-\mathbb{I}=z^{-1}\mathbf{E}_1 + O(z^{-2})$ as $z\to\infty$ where $\mathbf{E}_1=O(T^{-1})$.  
Unraveling the explicit transformations relating $\mathbf{Y}(\lambda)$ and $\mathbf{O}(z)=\mathbf{E}(z)\dot{\mathbf{O}}(z)$ and recalling the definition \eqref{eq:u-ucirc} of $u(x)$ gives, for both families $\mathrm{F}=\mathrm{gH}$ and $\mathrm{F}=\mathrm{gO}$,
\eq
u_\mathrm{F}^{[3]}(x;m,n)=u(x)=T^{\frac{1}{2}}U_\mathrm{F}^{[3]},\quad U_\mathrm{F}^{[3]}=-4s\frac{(\mathbf{E}(0)\dot{\mathbf{O}}^\mathrm{out}(0))_{11}(\mathbf{E}(0)\dot{\mathbf{O}}^\mathrm{out}(0))_{12}}{E_{1,12}+\dot{O}^\mathrm{out}_{1,12}},\quad x=T^{\frac{1}{2}}y_0+T^{-\frac{1}{2}}\zeta.
\endeq
%where for $|z|$ sufficiently large,
%\eq
%\dot{\mathbf{O}}(z)=\dot{\mathbf{O}}^\mathrm{out}(z)=\mathbb{I}+\sum_{k=1}^\infty z^{-k}\dot{\mathbf{O}}^\infty_k,
%\label{eq:OdotOut-elliptic-z-large}
%\endeq
%and for $|z|$ sufficiently small,
%\eq
%\dot{\mathbf{O}}(z)=\dot{\mathbf{O}}^\mathrm{out}(z)=\sum_{k=0}^\infty z^k\dot{\mathbf{O}}^0_k.
%\label{eq:OdotOut-elliptic-z-small}
%\endeq
Neglecting the small error terms, we therefore define an approximation for $U_\mathrm{F}^{[3]}$ by the formula
\eq
\begin{split}
\dot{U}^{[3]}_\mathrm{F}=\dot{U}^{[3]}_\mathrm{F}(\zeta;y_0):=&-4s\frac{\dot{O}^\mathrm{out}_{11}(0)\dot{O}^\mathrm{out}_{12}(0)}{\dot{O}^\mathrm{out}_{1,12}}\\
=&-4s\nu\frac{\dot{P}^\mathrm{out}_{11}(0)\dot{P}^\mathrm{out}_{12}(0)}{\dot{P}^\mathrm{out}_{1,12}},
\end{split}
\endeq
where in the second line we have used \eqref{eq:OP-infty-12} and \eqref{eq:OP-zero-product}.
Defining complex phase shifts by 
\eq
\begin{split}
\mathfrak{z}_1^{[3]}&:=-\ii a(0)-\ii a(z_0)\\
\mathfrak{z}_2^{[3]}&:=\ii a(0)-\ii a(z_0)\\
\mathfrak{p}_1^{[3]}&:=-\ii a(\infty)-\ii a(z_0)\\
\mathfrak{p}_2^{[3]}&:=\ii a(\infty)-\ii a(z_0),
\end{split}
\label{eq:dotU-phases}
\endeq
we use \eqref{eq:P-infty-12}--\eqref{eq:P-zero} and $\Theta(-z)=\Theta(z)$ to write $\dot{U}^{[3]}_\mathrm{F}(\zeta;y_0)$ in the form
\eq
\begin{split}
\dot{U}^{[3]}_\mathrm{F}(\zeta;y_0)&=
\psi^{[3]}_\mathrm{F}(y_0)
\frac{\Theta(a(0)+a(z_0)+K-\ii\varphi)\Theta(a(0)-a(z_0)-K+\ii\varphi)}{\Theta(a(\infty)+a(z_0)+K-\ii\varphi)\Theta(a(\infty)-a(z_0)-K+\ii\varphi)}\\
&=\psi_\mathrm{F}^{[3]}(y_0)\frac{\Theta(K-\ii(\varphi-\mathfrak{z}_1^{[3]}))\Theta(K-\ii(\varphi-\mathfrak{z}_2^{[3]}))}
{\Theta(K-\ii(\varphi-\mathfrak{p}_1^{[3]}))\Theta(K-\ii(\varphi-\mathfrak{p}_2^{[3]}))},
\end{split}
\label{eq:U-dot-define}
\endeq
which has the form \eqref{eq:intro-elliptic-theta} after using \eqref{eq:F1-U}, 
where
\eq
\psi_\mathrm{F}^{[3]}(y_0):=-4s\nu\frac{\mathcal{N}_{11}(y_0)\mathcal{N}_{12}(y_0)}{\mathcal{N}(y_0)}.  
\label{eq:psi-def}
\endeq
Note that using \eqref{eq:r-to-R}, \eqref{eq:z0-define}, \eqref{eq:calN-define}--\eqref{eq:calN-jk-define}, and the fact that $R(0)=4s$, the factor $\psi_\mathrm{F}^{[3]}(y_0)$ can be simplified to
\eq
\psi_\mathrm{F}^{[3]}(y_0)=z_0
\frac{\Theta(a(\infty)+a(z_0)+K)\Theta(a(\infty)-a(z_0)-K)}{\Theta(a(0)+a(z_0)+K)\Theta(a(0)-a(z_0)-K)}.
\label{eq:psi-simpler}
\endeq
It will be shown in Section~\ref{sec:Elliptic-DiffEq} below that $\dot{U}_\mathrm{F}^{[3]}(\zeta;y_0)$ is an elliptic function of $\zeta$.  We will have $U_\mathrm{F}^{[3]}=\dot{U}_\mathrm{F}^{[3]}+\bo(T^{-1})$ where $\dot{U}_\mathrm{F}^{[3]}$ is bounded and $1/U_\mathrm{F}^{[3]}=1/\dot{U}_\mathrm{F}^{[3]}+\bo(T^{-1})$ where $1/\dot{U}_\mathrm{F}^{[3]}$ is bounded provided that $(y_0,\zeta,T)\in\mathcal{S}(\epsilon)$ with $y_0$ in a compact subset of the relevant Boutroux domain, say $\mathcal{B}(\kappa)$, and that $\zeta$ is bounded (conditions guaranteeing the error estimates on $\mathbf{E}(z)$ according to Proposition~\ref{prop:G1-outer}).
%we also avoid the zeros of the denominator contributed by the factor $\Theta(a(\infty)-a(z_0)-K+\ii\varphi)$, points at which $\dot{O}^\mathrm{out}_{12}$ vanishes.  Equivalently, we may say that $U_\mathrm{F}^{[3]}=\dot{U}_\mathrm{F}^{[3]}+O(T^{-1})$ holds uniformly for $y_0$ in any compact subset of  $\mathcal{B}(\kappa)$ and $\zeta$ bounded, provided that $|\dot{U}_\mathrm{F}^{[3]}|\le M$ also holds for some fixed $M>0$.  

Similarly, for the rational function $u_\tw(x)$ we find that
\eq
u_\mathrm{F}^{[1]}(x;m,n)=u_\tw(x)=|\Theta_{0,\tw}|^\frac{1}{2}U_\mathrm{F}^{[1]},\quad
U_\mathrm{F}^{[1]}=
-\frac{T^\frac{1}{2}}{|\Theta_{0,\tw}|^\frac{1}{2}}\frac{(\mathbf{E}(0)\dot{\mathbf{O}}^\mathrm{out}(0))_{21}(E_{1,12}+\dot{O}^\mathrm{out}_{1,12})}{(\mathbf{E}(0)\dot{\mathbf{O}}^\mathrm{out}(0))_{11}}.
\endeq
To introduce the appropriate analogue of $\dot{U}_\mathrm{F}^{[3]}$ for this case, recall that $T=|\Theta_0|$ and
that, according to \eqref{eq:Baecklund-3-to-1}, $|\Theta_{0,\tw}|=\tfrac{1}{2}|\Theta_0|(1-s\kappa)$ with $1-s\kappa\neq 0$ for $\kappa\in (-1,1)$.  Therefore, we define
\eq
\dot{U}_\mathrm{F}^{[1]} := -\sqrt{\frac{2}{1-s\kappa}}\frac{\dot{O}^\mathrm{out}_{21}(0)\dot{O}^\mathrm{out}_{1,12}}{\dot{O}^\mathrm{out}_{11}(0)},
\label{eq:dotUtwist}
\endeq
which takes different forms depending on the sign $\nu=\pm 1$ in Table~\ref{tab:outer-uniformize-phases}: 
\eq
\begin{split}
\dot{U}_\mathrm{F}^{[1]}&=-\sqrt{\frac{2}{1-s\kappa}}\frac{\dot{P}^\mathrm{out}_{21}(0)\dot{P}^\mathrm{out}_{1,12}}{\dot{P}^\mathrm{out}_{11}(0)}\\
&=
-\sqrt{\frac{2}{1-s\kappa}}\frac{\mathcal{N}_{21}(y_0)\mathcal{N}(y_0)}{\mathcal{N}_{11}(y_0)}
\frac{\Theta(a(0)-a(z_0)-K-\ii\varphi)\Theta(a(\infty)-a(z_0)-K+\ii\varphi)}
{\Theta(a(0)+a(z_0)+K-\ii\varphi)\Theta(a(\infty)+a(z_0)+K+\ii\varphi)},\quad \nu=1;
\end{split}
\label{eq:dotUtwist-nu-positive}
\endeq
\eq
\begin{split}
\dot{U}_\mathrm{F}^{[1]} &=-\sqrt{\frac{2}{1-s\kappa}}\frac{\dot{P}^\mathrm{out}_{22}(0)\dot{P}^\mathrm{out}_{1,12}}{\dot{P}^\mathrm{out}_{12}(0)}\\
&=
-\sqrt{\frac{2}{1-s\kappa}}\frac{\mathcal{N}_{22}(y_0)\mathcal{N}(y_0)}{\mathcal{N}_{12}(y_0)}
\frac{\Theta(a(0)+a(z_0)+K+\ii\varphi)\Theta(a(\infty)-a(z_0)-K+\ii\varphi)}
{\Theta(a(0)-a(z_0)-K+\ii\varphi)\Theta(a(\infty)+a(z_0)+K+\ii\varphi)},\quad\nu=-1.
\end{split}
\label{eq:dotUtwist-nu-negative}
\endeq
In both cases, by similar arguments as used above to approximate $U_\mathrm{F}^{[3]}$ by $\dot{U}_\mathrm{F}^{[3]}$, we will have $U_\mathrm{F}^{[1]}=\dot{U}_\mathrm{F}^{[1]} + \bo(T^{-1})$ where $\dot{U}_\mathrm{F}^{[1]}$ is bounded and $1/U_\mathrm{F}^{[1]}=1/\dot{U}_\mathrm{F}^{[1]}+\bo(T^{-1})$ where $1/\dot{U}_\mathrm{F}^{[1]}$ is bounded, provided that $y_0$ lies in a compact subset of a Boutroux domain $\mathcal{B}(\kappa)$ for the given value $\kappa\in (-1,1)$, that $\zeta$ is bounded, and that $(y_0,\zeta,T)\in\mathcal{S}(\epsilon)$. The formul\ae\ \eqref{eq:dotUtwist-nu-positive}--\eqref{eq:dotUtwist-nu-negative} can be simplified and put into a universal form as follows.  Defining complex phase shifts by
\eq
\begin{split}
\mathfrak{z}_1^{[1]}&:= -\ii\nu a(0)+\ii a(z_0)\\
\mathfrak{z}_2^{[1]}&:= \ii a(\infty)-\ii a(z_0)\\
\mathfrak{p}_1^{[1]}&:= -\ii\nu a(0)-\ii a(z_0)\\
\mathfrak{p}_2^{[1]}&:= \ii a(\infty)+\ii a(z_0),
\end{split}
\label{eq:dotUtwist-phases}
\endeq
using $\Theta(-z)=\Theta(z)$ we can write
\eq
\begin{split}
\dot{U}_\mathrm{F}^{[1]} 
&= M \frac{\Theta(\nu a(0)-a(z_0)-K-\ii\varphi)\Theta(a(\infty)-a(z_0)-K+\ii\varphi)}{\Theta(\nu a(0)+a(z_0)+K-\ii\varphi)\Theta(a(\infty)+a(z_0)+K+\ii\varphi)}\\
&=M\frac{\Theta(K+\ii(\varphi-\mathfrak{z}_1^{[1]}))
\Theta(K-\ii(\varphi-\mathfrak{z}_2^{[1]}))}
{\Theta(K-\ii(\varphi-\mathfrak{p}_1^{[1]}))\Theta(K+\ii(\varphi-\mathfrak{p}_2^{[1]}))},
\end{split}
\endeq
where, using \eqref{eq:calN-define}--\eqref{eq:calN-jk-define},
\eq
M := -\sqrt{\frac{2}{1-s\kappa}}\frac{\alpha-\beta+\gamma-\delta}{4\ii}\frac{\Theta(a(\infty)+a(z_0)+K)}{\Theta(a(\infty)-a(z_0)-K)}\begin{cases}
\displaystyle \frac{f^\mathrm{OD}(0)}{f^\mathrm{D}(0)}\frac{\Theta(a(0)+a(z_0)+K)}{\Theta(a(0)-a(z_0)-K)},&\quad \nu=1\\
\displaystyle -\frac{f^\mathrm{D}(0)}{f^\mathrm{OD}(0)}\frac{\Theta(a(0)-a(z_0)-K)}{\Theta(a(0)+a(z_0)+K)},&\quad \nu=-1.
\end{cases}
\endeq
It is straightforward to use the definitions \eqref{eq:fD-fOD-define} and $j(z)^2(z-\beta)(z-\delta)=r(z)$ to confirm the identities
\eq
\frac{f^\mathrm{OD}(z)}{f^\mathrm{D}(z)}=\ii\frac{2r(z)-(z-\alpha)(z-\gamma)-(z-\beta)(z-\delta)}{(z-\alpha)(z-\gamma)-(z-\beta)(z-\delta)}\quad\text{and}\quad
\frac{f^\mathrm{D}(z)}{f^\mathrm{OD}(z)}=\ii\frac{2r(z)+(z-\alpha)(z-\gamma)+(z-\beta)(z-\delta)}{(z-\alpha)(z-\gamma)-(z-\beta)(z-\delta)}.
\endeq
Therefore, setting $z=0$, recalling \eqref{eq:z0-define} and using $r(0)=4s\nu$ we put $M$ in the universal form
\eq
M = -\sqrt{\frac{2}{1-s\kappa}}\frac{8s-(\alpha\gamma+\beta\delta)}{4z_0}
\frac{\Theta(a(\infty)+a(z_0)+K)\Theta(\nu a(0)+a(z_0)+K)}{\Theta(a(\infty)-a(z_0)-K)\Theta(\nu a(0)-a(z_0)-K)}.
\label{eq:psi-twist-simpler}
\endeq
Finally, using the identities in \eqref{eq:theta-identities} and $2K=2\pi\ii + H_\omega$ we have
\eq
\dot{U}^{[1]}_\mathrm{F}=\psi_\mathrm{F}^{[1]}\frac{\Theta(K-\ii(\varphi-\mathfrak{z}_1^{[1]}))
\Theta(K-\ii(\varphi-\mathfrak{z}_2^{[1]}))}
{\Theta(K-\ii(\varphi-\mathfrak{p}_1^{[1]}))\Theta(K-\ii(\varphi-\mathfrak{p}_2^{[1]}))},\quad
\psi_\mathrm{F}^{[1]}:=\ee^{\ii(\mathfrak{z}_1^{[1]}-\mathfrak{p}_2^{[1]})}M.
\label{eq:U-dot-twist-final}
\endeq
This also has the form \eqref{eq:intro-elliptic-theta} except that the theta function has the same parameter as in the approximation of the type-$3$ rational solutions of Painlev\'e-IV whereas we instead expect to see the theta function for a different elliptic curve associated to the type-$1$ parameters.  This final point will be clarified in Section~\ref{sec:related-Boutroux-curves} below; see \eqref{eq:equal-H-omegas}.
To use $\dot{U}_\mathrm{F}^{[1]}$ as an approximation of $U_\mathrm{F}^{[1]}=|\Theta_{0,\tw}|^{-\frac{1}{2}}u_\mathrm{F}^{[1]}(x;m,n)=|\Theta_{0,\mathrm{F}}^{[1]}(m,n)|^{-\frac{1}{2}}u_\mathrm{F}^{[1]}(x;m,n)$ for the type-$1$ function in the family $\mathrm{F}=\mathrm{gH}$ or $\mathrm{F}=\mathrm{gO}$, the variables and parameters in $\dot{U}_\mathrm{F}^{[1]}$ need to be carefully interpreted.  Here we recall Remark~\ref{rem:other-parameters}, which specifies that the parameters $T$, $s$, and $\kappa$ should be expressed in terms of the indices $(m,n)$ by \eqref{eq:gH-type1-RHP-parameters} or \eqref{eq:gO-type1-RHP-parameters} for the gH and gO families respectively.  Also, the variables $y_0$ and $\zeta$ should be rescaled by making the replacements \eqref{eq:y0zeta-type1-substitutions}--\eqref{eq:y0zeta-type1-substitutions-factors}.  Then finally we have a well-defined function $\dot{U}_\mathrm{F}^{[1]}=\dot{U}_\mathrm{F}^{[1]}(\zeta;y_0)$ where the arguments $\zeta$ and $y_0$ refer to the variables \emph{after} the indicated replacements have been made, and $\dot{U}_\mathrm{F}^{[1]}(\zeta;y_0)$ is an accurate approximation of $|\Theta_{0,\mathrm{F}}^{[1]}(m,n)|^{-\frac{1}{2}}u_\mathrm{F}^{[1]}(x;m,n)$ when $x=|\Theta_{0,\mathrm{F}}^{[1]}(m,n)|^\frac{1}{2}y_0+|\Theta_{0,\mathrm{F}}^{[1]}(m,n)|^{-\frac{1}{2}}\zeta$.

The original variable $y_0$ lies in a Boutroux domain $\mathcal{B}(\kappa)$ for $\kappa\in (-1,1)$, but after replacing the variables $y_0$ and $\zeta$ by their rescaled versions the approximation $U_\mathrm{F}^{[1]}=\dot{U}_\mathrm{F}^{[1]}(\zeta;y_0)+\bo(T^{-1})$ holds for $y_0$ in a homothetic dilation of $\mathcal{B}(\kappa)$.  Moreover, the rescaled domain should properly be associated to the leading term $I^{-s}(\kappa)=-(\kappa+3s)/(1-s\kappa)$ of $\kappa_\tw=-\Theta_{\infty,\tw}/|\Theta_{0,\tw}|=-\Theta_{\infty,\mathrm{F}}^{[1]}(m,n)/|\Theta_{0,\mathrm{F}}^{[1]}(m,n)|$, as the latter is the natural value of $\kappa$ associated to the rational solution $u^{[1]}_\mathrm{F}(x;m,n)$ according to the scalings in \eqref{eq:Thetas-scaling}.  Note that $\kappa\in (-1,1)$ implies that $|I^{-s}(\kappa)|>1$.
%Note that by \eqref{eq:y0-y0circ}, the condition that $y_0\in\mathcal{B}(\kappa)$ implies that $y_{0,\tw}$ lies within a homothetic dilation of that Boutroux domain, and the latter domain should properly be associated to the leading term $-(\kappa+3s)/(1-s\kappa)$ of $\kappa_\tw=-\Theta_{\infty,\tw}/|\Theta_{0,\tw}|$ rather than $\kappa\in (-1,1)$.  
Hence, referring to Definition~\ref{def:ExteriorDomains} we make the following definition.
\begin{definition}
Let $\rectangle=\rectangle(\kappa)$, $\TR=\TR(\kappa)$, $\TI=\TI(\kappa)$ be the Boutroux domains for $\kappa\in (-1,1)$ as explained in Sections~\ref{sec:rectangle-domain}, \ref{sec:TR}, and \ref{sec:TI} respectively.  If $\pm\kappa>1$, then $\rectangle(\kappa)$, $\TR(\kappa)$, and $\TI(\kappa)$ are defined instead by
\eq
\rectangle(\kappa):=\sqrt{\frac{1\pm\kappa}{2}}\rectangle(I^\pm(\kappa)),\quad
\TR(\kappa):=\sqrt{\frac{1\pm\kappa}{2}}\TR(I^\pm(\kappa)),\quad\text{and}\quad
\TI(\kappa):=\sqrt{\frac{1\pm\kappa}{2}}\TI(I^\pm(\kappa)),
\endeq
where the M\"obius transformations $I^\pm$ are given in \eqref{eq:DomainExtend} and $I^\pm(\kappa)\in (-1,1)$ for $\pm\kappa>1$.
\label{def:BoutrouxDomainsExtend}
\end{definition}

%If the solution of Riemann-Hilbert Problem~\ref{rhp:Pout} exists, then clearly the following convergent analogues of \eqref{eq:OdotOut-elliptic-z-large}--\eqref{eq:OdotOut-elliptic-z-small} hold:
%\eq
%\dot{\mathbf{P}}^\mathrm{out}(z;\zeta)=\mathbb{I}+\sum_{k=1}^\infty z^{-k}\dot{\mathbf{P}}^\infty_k(\zeta),
%\endeq
%\eq
%\dot{\mathbf{P}}^\mathrm{out}(z;\zeta)=\sum_{k=0}^\infty z^k\dot{\mathbf{P}}^0_k(\zeta).
%\endeq
%Then, taking into account the relation between $\dot{\mathbf{P}}^\mathrm{out}(z;\zeta)$ and $\dot{\mathbf{O}}^\mathrm{out}(z)$ as given in Table~\ref{tab:outer-uniformize}, the following formula for $\dot{U}=\dot{U}(\zeta)$ is equivalent to \eqref{eq:U-dot-define}:
%\eq
%\dot{U}(\zeta)=-4s\nu\frac{\dot{P}^0_{0,11}(\zeta)\dot{P}^0_{0,12}(\zeta)}{\dot{P}^\infty_{1,12}(\zeta)},
%\label{eq:dotUzeta-Pdot}
%\endeq
%where $\nu=\pm 1$ is given in each case in Table~\ref{tab:outer-uniformize-phases}.
%
%\textcolor{red}{Material to be merged in somehow:}
%
%From \eqref{eq:dotUzeta-Pdot}, we then find
%\eq
%\begin{split}
%\dot{U}(\zeta) = & \frac{16s\nu f^\text{D}(0)f^\text{OD}(0)}{\alpha-\beta+\gamma-\delta}\frac{\Theta(a(\infty)+a(z_0)+K)\Theta(a(\infty)-a(z_0)-K)}{\Theta(a(0)+a(z_0)+K)\Theta(a(0)-a(z_0)-K)} \\
%& \times \frac{\Theta(a(0)+a(z_0)+K-F_1(\zeta)H_\Omega)\Theta(a(0)-a(z_0)-K+F_1(\zeta)H_\Omega)}{\Theta(a(\infty)+a(z_0)+K-F_1(\zeta)H_\Omega)\Theta(a(\infty)-a(z_0)-K+F_1(\zeta)H_\Omega)}.
%\end{split}
%\endeq

\subsection{Differential equations satisfied by the approximations}
\label{sec:Elliptic-DiffEq}
\subsubsection{Derivation of the differential equation for $\dot{U}_\mathrm{F}^{[3]}(\zeta;y_0)$}
We now show that the function $\dot{U}(\zeta)=\dot{U}_\mathrm{F}^{[3]}(\zeta;y_0)$ defined by \eqref{eq:U-dot-define} satisfies exactly the differential equation \eqref{eq:elliptic-ODE} in which $E$ depends on $y_0$ via the Boutroux conditions \eqref{eq:Boutroux}.  Evaluating $\dot{\mathbf{O}}^\mathrm{out}(z)$ at $z=0$ yields a matrix function of $\zeta$ that we will write as $\dot{\mathbf{O}}^\mathrm{out}(0)=\mathbf{Z}(\zeta)$ in this section.  Likewise, to emphasize the dependence on $\zeta$ in the expansion coefficients in \eqref{eq:OdotOut-elliptic-z-large} we will write $\dot{\mathbf{O}}^\mathrm{out}_k=\dot{\mathbf{O}}^\mathrm{out}_k(\zeta)$.  

Fixing $y_0\in \rectangle\cup\TR\cup\TI$, we observe that the matrix $\mathbf{F}(\zeta;z):=\dot{\mathbf{O}}^\mathrm{out}(z)\ee^{\zeta z\sigma_3/2}$ has jump matrices that are independent of $\zeta$, so since $\det(\mathbf{F}(\zeta;z))=1$, $\mathbf{F}'(\zeta;z)\mathbf{F}(\zeta;z)^{-1}$ is an entire function of $z$.  Using the convergent Laurent expansion \eqref{eq:OdotOut-elliptic-z-large} and its term-by-term derivative with respect to $\zeta$ shows that in fact this entire function is linear in $z$:  $\mathbf{F}'(\zeta;z)\mathbf{F}(\zeta;z)^{-1}=\tfrac{1}{2}z\sigma_3+\tfrac{1}{2}[\dot{\mathbf{O}}^\mathrm{out}_1(\zeta),\sigma_3]$.  Equivalently, for each $z$ not on the jump contour, the outer parametrix satisfies the differential equation
\eq
\frac{\dd\dot{\mathbf{O}}^\mathrm{out}}{\dd\zeta}=\frac{1}{2}z[\sigma_3,\dot{\mathbf{O}}^\mathrm{out}]+\frac{1}{2}[\dot{\mathbf{O}}^\mathrm{out}_1(\zeta),\sigma_3]\dot{\mathbf{O}}^\mathrm{out}.
\label{eq:OdotOut-DE}
\endeq
In particular, upon setting $z=0$, $\dot{\mathbf{O}}^\mathrm{out}$ becomes $\mathbf{Z}(\zeta)$, and we deduce that \eqref{eq:OdotOut-DE} implies that
\eq
\frac{\dd}{\dd\zeta}(Z_{11}(\zeta)Z_{12}(\zeta))=-\dot{O}^\mathrm{out}_{1,12}(\zeta)(Z_{11}(\zeta)Z_{22}(\zeta)+Z_{12}(\zeta)Z_{21}(\zeta)).
\label{eq:OdotOut-zero-DE}
\endeq
Similarly, using again the Laurent expansion \eqref{eq:OdotOut-elliptic-z-large} and taking the terms in \eqref{eq:OdotOut-DE} proportional to $z^{-1}$ we find that
\eq
\frac{\dd\dot{O}^\mathrm{out}_{1,12}(\zeta)}{\dd \zeta}=\dot{O}^\mathrm{out}_{2,12}(\zeta)-\dot{O}^\mathrm{out}_{1,12}(\zeta)\dot{O}^\mathrm{out}_{1,22}(\zeta).
\label{eq:OdotOut-infty-DE}
\endeq

Next, we make the following observation:  if $\mathbf{V}(z)$ is the jump matrix for the outer parametrix $\dot{\mathbf{O}}^\mathrm{out}(z)$, then on arcs of the jump contour where $R(z)$ is continuous we have the form $\mathbf{V}(z)=\mathbf{D}(a)$ for some constant $a\neq 0$ and therefore $\mathbf{V}(z)\sigma_3\mathbf{V}(z)^{-1}=\sigma_3$, while on arcs of the jump contour across which $R(z)$ changes sign we have instead the $\mathbf{V}(z)=\mathbf{T}(a\ee^{-\zeta z})$ for some constant $a\neq 0$ and therefore $\mathbf{V}(z)\sigma_3\mathbf{V}(z)^{-1}=-\sigma_3$. It follows that the matrix function 
\eq
\mathbf{G}(z):= R(z)\dot{\mathbf{O}}^{\mathrm{out}}(z)\sigma_3\dot{\mathbf{O}}^{\mathrm{out}}(z)^{-1}
\label{eq:G-matrix-def}
\endeq
is analytic except possibly on the jump contour, on which it is continuous except possibly for the endpoints of each maximal arc.  Those endpoints are the roots of the quartic polynomial $P(z)=R(z)^2$, and since $\dot{\mathbf{O}}^{\mathrm{out}}(z)$ blows up at these points like a negative one-fourth power while $\det(\dot{\mathbf{O}}^{\mathrm{out}}(z))=1$, it follows by Morera's Theorem that $\mathbf{G}(z)$ is an entire function of $z$.  From the asymptotic behavior of the factors (see \eqref{eq:OdotOut-elliptic-z-large}), it is clear that $\mathbf{G}(z)$ is in fact a quadratic matrix-valued polynomial in $z$.  Using \eqref{eq:OdotOut-elliptic-z-large} and the expansion $R(z)=z^2 + 2y_0z+4\kappa + \bo(z^{-1})$ as $z\to\infty$ to calculate the polynomial part of the right-hand side of \eqref{eq:G-matrix-def} gives the representation
\eq
\mathbf{G}(z)=\sigma_3 z^2 + (2y_0\sigma_3+[\dot{\mathbf{O}}^\mathrm{out}_1(\zeta),\sigma_3])z + \mathbf{G}(0),
\label{eq:G-poly}
\endeq
where
\eq
\mathbf{G}(0):=4\kappa\sigma_3+2y_0[\dot{\mathbf{O}}^\mathrm{out}_1(\zeta),\sigma_3]+[\sigma_3,\dot{\mathbf{O}}^\mathrm{out}_1(\zeta)]\dot{\mathbf{O}}^\mathrm{out}_1(\zeta) + [\dot{\mathbf{O}}^\mathrm{out}_2(\zeta),\sigma_3].
\label{eq:G(0)-z-large}
\endeq
On the other hand, setting $z=0$ on the right-hand side of \eqref{eq:G-matrix-def} and using $R(0)=4s$ 
% and the notation from \eqref{eq:OdotOut-elliptic-z-small} 
gives an equivalent representation for $\mathbf{G}(0)$:
\eq
\mathbf{G}(0)=4s\mathbf{Z}(\zeta)\sigma_3\mathbf{Z}(\zeta)^{-1}.
\label{eq:G(0)-z-small}
\endeq
Comparing the $(1,2)$-entry in the equivalent representations \eqref{eq:G(0)-z-large}--\eqref{eq:G(0)-z-small} gives the identity 
\eq
\dot{O}^\mathrm{out}_{2,12}(\zeta)-\dot{O}^\mathrm{out}_{1,12}(\zeta)\dot{O}^\mathrm{out}_{1,22}(\zeta)=4sZ_{11}(\zeta)Z_{12}(\zeta)-2y_0\dot{O}^\mathrm{out}_{1,12}(\zeta).
\endeq
Using this identity and combining \eqref{eq:OdotOut-zero-DE}--\eqref{eq:OdotOut-infty-DE} with the definition \eqref{eq:U-dot-define} shows that
\eq
\begin{split}
\dot{U}'(\zeta)&=\dot{U}(\zeta)^2 + 2y_0\dot{U}(\zeta)+4s(Z_{11}(\zeta)Z_{22}(\zeta)+Z_{12}(\zeta)Z_{21}(\zeta))\\
&= \dot{U}(\zeta)^2 + 2y_0\dot{U}(\zeta) + G_{11}(0),
\end{split}
\label{eq:U-dot-DE-1}
\endeq
where on the second line we used the $(1,1)$-entry of \eqref{eq:G(0)-z-small}.  Therefore also
\eq
\begin{split}
\dot{U}'(\zeta)^2 &= \dot{U}(\zeta)^4 + 4y_0\dot{U}(\zeta)^3 +(4y_0^2+2G_{11}(0))\dot{U}(\zeta)^2 + 4y_0G_{11}(0)\dot{U}(\zeta)+G_{11}(0)^2\\
&=\dot{U}(\zeta)^4 + 4y_0\dot{U}(\zeta)^3 +(4y_0^2+2G_{11}(0))\dot{U}(\zeta)^2 + 4y_0G_{11}(0)\dot{U}(\zeta)+16-G_{12}(0)G_{21}(0),
\end{split}
\label{eq:U-dot-DE-2}
\endeq
where on the second line we used that $\mathrm{tr}(\mathbf{G}(0))=0$ and $\det(\mathbf{G}(0))=-16$, both of which follow from \eqref{eq:G(0)-z-small}.

Now, since $\sigma_3^2=\mathbb{I}$ the definition \eqref{eq:G-matrix-def} shows that $\mathbf{G}(z)^2$ is a \emph{scalar} polynomial, namely $\mathbf{G}(z)^2=R(z)^2\mathbb{I}=P(z)\mathbb{I}$.  Squaring \eqref{eq:G-poly} and taking (without loss of generality) the $(1,1)$-entry gives
\begin{multline}
P(z)=z^4 + 4y_0z^3 + (4y_0^2+2G_{11}(0)-4\dot{O}^\mathrm{out}_{1,12}(\zeta)\dot{O}^\mathrm{out}_{1,21}(\zeta))z^2 \\{}+ (4y_0G_{11}(0)+2\dot{O}^\mathrm{out}_{1,21}(\zeta)G_{12}(0)-2\dot{O}^\mathrm{out}_{1,12}(\zeta)G_{21}(0))z+16.
\end{multline}
Substituting $z=\dot{U}(\zeta)$ and subtracting from \eqref{eq:U-dot-DE-2} gives
\eq
\dot{U}'(\zeta)^2-P(\dot{U}(\zeta))=4\dot{O}^\mathrm{out}_{1,12}(\zeta)\dot{O}^\mathrm{out}_{1,21}(\zeta)\dot{U}(\zeta)^2+2(\dot{O}^\mathrm{out}_{1,12}(\zeta)G_{21}(0)-\dot{O}^\mathrm{out}_{1,21}(\zeta)G_{12}(0))\dot{U}(\zeta) - G_{12}(0)G_{21}(0).
\endeq
Recalling the definition \eqref{eq:U-dot-define} of $\dot{U}(\zeta)$ and using the matrix elements of \eqref{eq:G(0)-z-small} then shows that the right-hand side vanishes identically in $\zeta$, which completes the proof of the claim.  Hence $\dot{U}(\zeta)=\dot{U}^{[3]}_\mathrm{F}(\zeta;y_0)$ can be written in the form $f(\zeta-\zeta_0)$ for some $\zeta_0$ independent of $\zeta$, where $f(\zeta)$ is the unique solution of the differential equation \eqref{eq:elliptic-ODE} satisfying $f(0)=0$ and $f'(0)=4$.

\subsubsection{Derivation of the differential equation for $\dot{U}_\mathrm{F}^{[1]}(\zeta;y_0)$}
For convenience, let us relabel the arguments of $\dot{U}_\mathrm{F}^{[1]}$ as $\zeta_\tw$ and $y_{0,\tw}$.
We now show that $\dot{U}_\tw(\zeta_\tw)=\dot{U}_\mathrm{F}^{[1]}(\zeta_\tw;y_{0,\tw})$ is also an elliptic function of its argument $\zeta_\tw$, solving a closely related differential equation.  To see this, we start from the definition \eqref{eq:U-dot-define} of $\dot{U}(\zeta)=\dot{U}^{[3]}_\mathrm{F}(\zeta;y_0)$ and use \eqref{eq:OdotOut-zero-DE} (with $\det(\mathbf{Z}(\zeta))=1$ on the right-hand side to eliminate $Z_{11}(\zeta)Z_{22}(\zeta)$) and \eqref{eq:OdotOut-infty-DE} to find the differential identity
\eq
\frac{1}{2\dot{U}(\zeta)}\frac{\dd\dot{U}}{\dd\zeta}(\zeta)-\frac{2s}{\dot{U}(\zeta)}-y_0-\frac{1}{2}\dot{U}(\zeta) = -\frac{Z_{21}(\zeta)\dot{O}^\mathrm{out}_{1,12}(\zeta)}{Z_{11}(\zeta)} = 
\sqrt{\frac{1-s\kappa}{2}}\dot{U}_\tw(\zeta_\tw),
\label{eq:dotU-dotUtwist}
\endeq
where in the second equality we used the definition \eqref{eq:dotUtwist}.  Now, using the fact shown above that $\dot{U}(\zeta)$ satisfies the first-order equation \eqref{eq:elliptic-ODE} and hence (by isolating the constant $E$ and taking a derivative) the second-order equation \eqref{eq:Approximating-ODE-Second-Order}, it is straightforward to check that $\dot{U}_\tw$ also satisfies a related second-order equation:
\eq
\frac{\dd^2\dot{U}_\tw}{\dd\zeta^2}=\frac{1}{2\dot{U}_\tw}\left(\frac{\dd\dot{U}_\tw}{\dd\zeta}\right)^2 +\frac{3}{4}(1-s\kappa)\dot{U}_\tw^3+4y_0\sqrt{\frac{1-s\kappa}{2}}\dot{U}_\tw^2 + (2y_0^2-6s-2\kappa)\dot{U}_\tw - \frac{4(1-s\kappa)}{\dot{U}_\tw}.
\endeq
But now recall that $\dot{U}_\tw$ should be considered as a function of $\zeta_\tw=\sqrt{\tfrac{1}{2}(1-s\kappa)}\zeta$, so by the chain rule,
\eq
\frac{\dd^2\dot{U}_\tw}{\dd\zeta_\tw^2}=\frac{1}{2\dot{U}_\tw}
\left(\frac{\dd\dot{U}_\tw}{\dd\zeta_\tw}\right)^2 + \frac{3}{2}\dot{U}_\tw^3 + 4y_0
\sqrt{\frac{2}{1-s\kappa}}
\dot{U}_\tw^2 + \left(2y_0^2
\frac{2}{1-s\kappa}
+4\left[-\frac{\kappa+3s}{1-s\kappa}\right]\right)\dot{U}_\tw-\frac{8}{\dot{U}_\tw}.
\label{eq:twisted-second-order}
\endeq
%To derive \eqref{eq:twisted-second-order} from \eqref{eq:dotU-dotUtwist} and \eqref{eq:Approximating-ODE-Second-Order} one must be careful to account for the chain rule in differentiating $\dot{U}$ because its argument $\zeta$ is a rescaling of the argument $\zeta$ of $\dot{U}_\tw$ according to \eqref{eq:y0zeta-type1-substitutions} and \eqref{eq:theta0-ratio}.  Likewise, the parametric variable $y_0$ in \eqref{eq:twisted-second-order} is that which is an argument of $\dot{U}_\tw$, and hence every occurrence of $y_0$ in \eqref{eq:dotU-dotUtwist} and \eqref{eq:Approximating-ODE-Second-Order} must also be rescaled accordingly.  
The differential equation \eqref{eq:twisted-second-order} matches exactly the form of \eqref{eq:Approximating-ODE-Second-Order} in which only 
$y_0$ is replaced with $y_{0,\tw}$ determined by the relation $y_0=\sqrt{\tfrac{1}{2}(1-s\kappa)}y_{0,\tw}$ and 
$\kappa$ is replaced with with $I^{-s}(\kappa)=-(\kappa+3s)/(1-s\kappa)$.  Therefore, as in Section~\ref{sec:nonequilibrium}, this equation can be integrated up to the form
\eq
\left(\frac{\dd\dot{U}_\tw}{\dd\zeta_\tw}\right)^2=P_\tw(\dot{U}_\tw):=\dot{U}_\tw^4 + 4
y_{0,\tw}
\dot{U}_\tw^3 + 2(2
y_{0,\tw}^2
 + 4I^{-s}(\kappa))\dot{U}_\tw^2 + 2E_\tw\dot{U}_\tw + 16,
\label{eq:elliptic-ODE-twist}
\endeq
which should be compared with \eqref{eq:elliptic-ODE}.  Here $E_\tw$ is a constant of integration.  Solving for $E_\tw$, eliminating $\dot{U}_\tw$ and its derivative in favor of derivatives of $\dot{U}(\zeta)$ using \eqref{eq:dotU-dotUtwist}, and finally using the differential equations \eqref{eq:Approximating-ODE-Second-Order} and \eqref{eq:elliptic-ODE} satisfied by $\dot{U}(\zeta)$ (and the chain rule again) one finds the explicit relation between $E_\tw$ and $E$:
\eq
E_\tw = \left(\frac{2}{1-s\kappa}\right)^\frac{3}{2}\left[E-4y_0(\kappa+s)\right].
\label{eq:E-twist}
\endeq

\begin{remark}
Observe that the relation \eqref{eq:dotU-dotUtwist} expressing $\dot{U}_\tw$ explicitly in terms of $\dot{U}$ is a limiting form of the non-isomonodromic B\"acklund transformation \eqref{eq:Baecklund-3-to-1} expressing $u_\tw(x)$ explicitly in terms of $u(x)$.
\label{rem:limitingBaecklund}
\end{remark}

\subsubsection{Relation between the corresponding spectral curves}
\label{sec:related-Boutroux-curves}
If we use the differential equation \eqref{eq:elliptic-ODE} to eliminate the derivative in \eqref{eq:dotU-dotUtwist}, we obtain a non-differential relation between $\dot{U}$ and $\dot{U}_\tw$:
\eq
\dot{U}_\tw=-\sqrt{\frac{2}{1-s\kappa}}\left(\frac{1}{2}\dot{U}+y_0+\frac{2s}{\dot{U}}+\frac{1}{2}\frac{w}{\dot{U}}\right),\quad w^2=P(\dot{U}).
\label{eq:dotUtwist-algebraic}
\endeq
Isolating $w$ and squaring leads to the bi-quadratic relation
\eq
2\sqrt{\frac{1-s\kappa}{2}}\dot{U}_\tw\dot{U}^2 + \left[4s(1-s\kappa)+4y_0\sqrt{\frac{1-s\kappa}{2}}\dot{U}_\tw+(1-s\kappa)\dot{U}_\tw^2\right]\dot{U} + \left[8s\sqrt{\frac{1-s\kappa}{2}}\dot{U}_\tw + 8sy_0 - E\right]=0,
\endeq
where we used the fact that $\dot{U}$ does not vanish identically to cancel a common factor.  Solving now instead for $\dot{U}$ and using \eqref{eq:E-twist} and $y_0=\sqrt{\tfrac{1}{2}(1-s\kappa)}y_{0,\tw}$, we obtain
\eq
\dot{U}=-\sqrt{\frac{1-s\kappa}{2}}\left(\frac{1}{2}\dot{U}_\tw+y_{0,\tw} + \frac{2s}{\dot{U}_\tw} +\frac{1}{2}\frac{w_\tw}{\dot{U}_\tw}\right),\quad w_\tw^2=P_\tw(\dot{U}_\tw).
\label{eq:dotU-algebraic}
\endeq
Solving \eqref{eq:dotU-algebraic} for $w_\tw$,  eliminating $\dot{U}_\tw$ using \eqref{eq:dotUtwist-algebraic} and expressing $y_{0,\tw}$ in terms of $y_0$ gives
\eq
w_\tw=-\frac{1}{2(1-s\kappa)\dot{U}^2}\left[w^2-3\dot{U}^4-2w\dot{U}^2-8y_0\dot{U}^3 + 8sw-4(y_0^2+2\kappa)\dot{U}^2+16\right].
\label{eq:wtwist-algebraic}
\endeq
Likewise, solving \eqref{eq:dotUtwist-algebraic} for $w$, eliminating $\dot{U}$ using \eqref{eq:dotU-algebraic} and expressing $y_0$ in terms of $y_{0,\tw}$ gives
\eq
w=-\frac{1-s\kappa}{8\dot{U}_\tw^2}\left[w_\tw^2-3\dot{U}_\tw^4-2w_\tw\dot{U}_\tw^2-8y_{0,\tw}\dot{U}_\tw^3+8sw_\tw-4(y_{0,\tw}^2+2I^{-s}(\kappa))\dot{U}_\tw^2 + 16\right].
\label{eq:w-algebraic}
\endeq
Using $s=\pm 1$, it is straightforward to check that the equations \eqref{eq:dotUtwist-algebraic} and \eqref{eq:wtwist-algebraic} constitute a birational transformation $\mathcal{T}:\mathbb{C}^2\to\mathbb{C}^2$ with action $(\dot{U}_\tw,w_\tw)=\mathcal{T}(\dot{U},w)$ and with explicit inverse given by \eqref{eq:dotU-algebraic} and \eqref{eq:w-algebraic}.  Restricting to the Riemann surface $\mathcal{R}$ shows that $\mathcal{T}:\mathcal{R}\to\mathcal{R}_\tw$ and that $\mathcal{T}^{-1}:\mathcal{R}_\tw\to\mathcal{R}$.  Therefore the Riemann surfaces $\mathcal{R}$ and $\mathcal{R}_\tw$ are conformally equivalent.  In particular, the pull-back under $\mathcal{T}$ of the holomorphic differential $w_\tw^{-1}\dd\dot{U}_\tw$ on $\mathcal{R}_\tw$ is a holomorphic differential on $\mathcal{R}$:  taking the differential of \eqref{eq:dotUtwist-algebraic} and using $2w\,\dd w=P'(\dot{U})\,\dd\dot{U}$ and $\dot{U}P'(\dot{U})=w^2 + 3\dot{U}^4+8y_0\dot{U}^3+4(y_0^2+2\kappa)\dot{U}^2-16$ gives
\eq
\begin{split}
\mathcal{T}^*\frac{\dd\dot{U}_\tw}{w_\tw}&=-\sqrt{\frac{2}{1-s\kappa}}\left(\frac{1}{2}-\frac{2s}{\dot{U}^2}-\frac{w}{2\dot{U}^2}+\frac{P'(\dot{U})}{4w\dot{U}}\right)\,\frac{\dd\dot{U}}{w_\tw}\\
&=\sqrt{\frac{2}{1-s\kappa}}\frac{w^2-3\dot{U}^4-2w\dot{U}^2-8y_0\dot{U}^3+8sw-4(y_0^2+2\kappa)\dot{U}^2+16}{4ww_\tw\dot{U}^2}\,\dd\dot{U}\\
&= -\sqrt{\frac{1-s\kappa}{2}}\frac{\dd\dot{U}}{w}
\end{split}
\endeq
after comparing with \eqref{eq:wtwist-algebraic}.  Therefore, if $(\mathfrak{a},\mathfrak{b})$ is a canonical homology basis on $\mathcal{R}$, then $(\mathfrak{a}_\tw,\mathfrak{b}_\tw)=(\mathcal{T}\mathfrak{a},\mathcal{T}\mathfrak{b})$ is a canonical homology basis on $\mathcal{R}_\tw$, and the corresponding theta function parameters are equal:
\eq
H_{\omega,\tw}=2\pi\ii\frac{\displaystyle\oint_{\mathfrak{b}_\tw}\frac{\dd\dot{U}_\tw}{w_\tw}}{
\displaystyle\oint_{\mathfrak{a}_\tw}\frac{\dd\dot{U}_\tw}{w_\tw}} = 
2\pi\ii\frac{\displaystyle\oint_\mathfrak{b}\frac{\dd\dot{U}}{w}}{\displaystyle\oint_\mathfrak{a}\frac{\dd\dot{U}}{w}}=H_\omega.
\label{eq:equal-H-omegas}
\endeq

Likewise, if $C$ denotes any cycle on the Riemann surface $\mathcal{R}$ of the equation $w^2=P(\dot{U})$, then by integration by parts
\eq
\oint_C \dot{U}_\tw\,\dd \dot{U} = -\oint_{C_\tw}\dot{U}\,\dd \dot{U}_\tw,
\endeq
where $C_\tw$ is the corresponding cycle on the Riemann surface $\mathcal{R}_\tw$ of $w_\tw^2=P_\tw(\dot{U}_\tw)$.  Therefore, using \eqref{eq:dotUtwist-algebraic} and \eqref{eq:dotU-algebraic} we get
\eq
-\sqrt{\frac{2}{1-s\kappa}}\oint_C\left(\frac{1}{2}\dot{U}+y_0+\frac{2s}{\dot{U}}+\frac{1}{2}\frac{w}{\dot{U}}\right)\dd\dot{U}=
\sqrt{\frac{1-s\kappa}{2}}\oint_{C_\tw}\left(\frac{1}{2}\dot{U}_\tw+y_{0,\tw}+\frac{2s}{\dot{U}_\tw}
+\frac{1}{2}\frac{w_\tw}{\dot{U}_\tw}\right)\,\dd\dot{U}_\tw.
\endeq
The first two terms in the integrand on each side of this equation contribute nothing by Cauchy's Theorem.  The third term on each side can only contribute a purely imaginary quantity as the residue of the pole is real in each case.  Therefore, taking the real part we obtain
\eq
\sqrt{\frac{2}{1-s\kappa}}\mathrm{Re}\left(\oint_C\frac{w}{\dot{U}}\,\dd\dot{U}\right)=-\sqrt{\frac{1-s\kappa}{2}}\mathrm{Re}\left(\oint_{C_\tw}\frac{w_\tw}{\dot{U}_\tw}\,\dd\dot{U}_\tw\right).
\label{eq:twist-Boutroux-identity}
\endeq
It follows that $\mathcal{R}$ with parameter $E$ is a Boutroux curve if and only if $\mathcal{R}_\tw$ with parameter $E_\tw$ is a Boutroux curve.  This finally motivates the following definition of $E=E(y_0;\kappa)$ for all $\kappa\in\mathbb{R}\setminus\{-1,1\}$.
\begin{definition}[Parameter $E=E(y_0;\kappa)$ for Boutroux curves]
If $\kappa\in (-1,1)$, $E=E(y_0;\kappa)$ is defined as explained in Sections~\ref{sec:rectangle-domain}--\ref{sec:TI}, extended to $y_0\in (-\TR(\kappa))\cup(-\TI(\kappa))$ by odd reflection.  If instead $\pm\kappa>1$, then $E(y_0;\kappa)$ is defined by
\eq
E(y_0;\kappa):=\left(\frac{2}{1\pm\kappa}\right)^\frac{3}{2}\left[E\left(\sqrt{\tfrac{1}{2}(1\pm\kappa)}y_0;I^\pm(\kappa)\right)-4(\kappa\mp 1)\sqrt{\tfrac{1}{2}(1\pm\kappa)}y_0\right],\quad I^\pm(\kappa):=-\frac{\kappa\mp 3}{1\pm\kappa},\quad\pm\kappa>1.
\endeq
The M\"obius transformations $I^\pm$ are both involutions and $I^+:(1,+\infty)\to (-1,1)$ while $I^-:(-\infty,-1)\to (-1,1)$.
\label{def:E}
\end{definition}
We now give the proof of Proposition~\ref{prop:E}.
\begin{proof}[Proof of Proposition~\ref{prop:E}]
That $E(y_0;\kappa)$ satisfies the Boutroux equations \eqref{eq:intro-Boutroux} if $\kappa\in (-1,1)$ has already been established in Sections~\ref{sec:rectangle-domain}--\ref{sec:TI}.  The identity \eqref{eq:twist-Boutroux-identity} then shows that when the definition of $E$ is extended to $\pm \kappa>1$ as indicated in Definition~\ref{def:E} the same holds for $|\kappa|>1$.  Similarly, the smoothness and continuity properties of $E$ already established for $\kappa\in (-1,1)$ directly extend to $|\kappa|>1$ by the definition.  The symmetry \eqref{eq:E-twist-symmetry} is just a direct consequence of Definition~\ref{def:E}, while \eqref{eq:E-flip-symmetry} holds for $\kappa\in (-1,1)$ as was shown in Sections~\ref{sec:rectangle-domain}--\ref{sec:TI} and then extends to $|\kappa|>1$ again just from the definition.
\end{proof}

Finally, we have shown that, like $\dot{U}^{[3]}_\mathrm{F}(\zeta;y_0)$, $\dot{U}^{[1]}_\mathrm{F}(\zeta;y_0)$ satisfies the differential equation \eqref{eq:elliptic-ODE} for a value of $\kappa$ very close (see Remark~\ref{rem:SlightlyDeformedSpectralCurve} below) to its ``native'' value of $-\Theta_{\infty,\mathrm{F}}^{[1]}(m,n)/|\Theta_{0,\mathrm{F}}^{[1]}(m,n)|$ with corresponding parameter $E=E(y_0;\kappa)$ chosen so that the underlying Riemann surface is a Boutroux curve.  Hence it can be written in the form $f(\zeta-\zeta_0)$ for $\zeta_0$ independent of $\zeta$, where again $f(\zeta)$ is the unique solution of \eqref{eq:elliptic-ODE} with $f(0)=0$ and $f'(0)=4$.
Note that in the gO case, the approximations defined above for $y_0\in\TR(\kappa)$ or $y_0\in\TI(\kappa)$ are easily extended to $-y_0\in\TR(\kappa)$ and $-y_0\in\TI(\kappa)$ by odd reflection:  $(\dot{U},y_0,\zeta)\mapsto (-\dot{U},-y_0,-\zeta)$, which is consistent with the exact symmetry of the rational solutions as indicated in Proposition~\ref{prop:FirstQuadrant}.
To get an approximation in the only remaining case of type $j=2$, we use the symmetry \eqref{eq:symmetry-1-2} to express $u^{[2]}_\mathrm{F}(\cdot;m,n)$ in terms of $u^{[1]}_\mathrm{F}(\cdot;n,m)$, rotating both $y_0$ and $\zeta$ by a quarter turn in the complex plane:  $\dot{U}^{[2]}_\mathrm{F}(\zeta;y_0):=\ii\dot{U}^{[1]}_\mathrm{F}(-\ii\zeta;-\ii y_0)$.  Note that the approximating differential equation \eqref{eq:elliptic-ODE} is invariant under $(f,\zeta,y_0,\kappa)\mapsto (-\ii f,-\ii\zeta,-\ii y_0,-\kappa)$ given the symmetry \eqref{eq:E-flip-symmetry} in Proposition~\ref{prop:E}.  Hence the type $j=2$ approximation is also an elliptic function of the form $f(\zeta-\zeta_0)$.  

\begin{remark}
A subtle point is that when $\kappa$ and $s$ are defined by \eqref{eq:gH-type1-RHP-parameters} or \eqref{eq:gO-type1-RHP-parameters} as necessary to study $u_\mathrm{gH}^{[1]}(x;m,n)$ or $u_\mathrm{gO}^{[1]}(x;m,n)$ respectively according to Remark~\ref{rem:other-parameters}, then the parameter $I^{-s}(\kappa)$ appearing in the differential equation 
\eqref{eq:elliptic-ODE-twist} satisfied by the approximation $\dot{U}^{[1]}_\mathrm{F}(\zeta;y_0)$ is not exactly equal to the value of $\kappa$ for which an elliptic function solution of the differential equation \eqref{eq:elliptic-ODE} is asserted in Theorems~\ref{thm:Hermite-elliptic} and \ref{thm:Okamoto-elliptic} as a valid approximation.  The latter value is the ratio $-\Theta_{\infty,\mathrm{F}}^{[1]}(m,n)/|\Theta_{0,\mathrm{F}}^{[1]}(m,n)|$.  In fact for the gH case we have $s=1$ and taking $\kappa$ in terms of $m,n$ from \eqref{eq:gH-type1-RHP-parameters},
\eq
-\frac{\Theta_{\infty,\mathrm{gH}}^{[1]}(m,n)}{|\Theta_{0,\mathrm{gH}}^{[1]}(m,n)|}=I^{-s}(\kappa)-\frac{2}{n}
\endeq
while for the gO case we may take either sign for $s$ and taking $\kappa$ from \eqref{eq:gO-type1-RHP-parameters},
\eq
-\frac{\Theta_{\infty,\mathrm{gO}}^{[1]}(m,n)}{|\Theta_{0,\mathrm{gO}}^{[1]}(m,n)|}=I^{-s}(\kappa)-\frac{2s}{n-\frac{1}{3}}.
\endeq
Arriving at the precise statements in Theorems~\ref{thm:Hermite-elliptic} and \ref{thm:Okamoto-elliptic} respectively therefore involves an additional approximation in which one elliptic function is exchanged for another with periods differing by $\bo(n^{-1})$ which is equivalent to $\bo(T^{-1})$.  Since $\zeta$ is bounded, this perturbation can be absorbed into the error terms, although the phase $\zeta_0$ generally needs to be shifted by an amount that does not tend to zero in the limit $T\to\infty$ because it involves terms proportional to $T$.  Note however that the statements of these two theorems do not precisely specify the value of $\zeta_0$.  A similar minor discrepancy arises in the approximation of $u_\mathrm{F}^{[2]}(x;m,n)$ because the symmetry $\kappa\mapsto -\kappa$ does not exactly correspond to the Boiti-Pempinelli symmetry (see \eqref{eq:symmetry-1-2}) because the latter is a reflection through the horizontal line $\Theta_\infty=\tfrac{1}{2}$ instead of through $\Theta_\infty=0$.  This discrepancy is dealt with in exactly the same way.
\label{rem:SlightlyDeformedSpectralCurve}
\end{remark}

To complete the proof of Theorems~\ref{thm:Hermite-elliptic} and \ref{thm:Okamoto-elliptic}, it only remains to note that suitable uniformity of the error estimates follows from similar arguments as given in Section~\ref{sec:Exterior-Uniformity} and then to explain why the condition that $(y_0,\zeta,T)\in\mathcal{S}(\epsilon)$ for some $\epsilon>0$ can be dropped.   This second point will be justified at the end of Section~\ref{sec:MalgrangeResidues} below.

\subsection{Zeros and poles of the approximations}
The approximations $\dot{U}^{[j]}_\mathrm{F}(\zeta;y_0)$ of the rational solutions of type $j=1,3$ in the family $\mathrm{F}=\mathrm{gH}$ or $\mathrm{F}=\mathrm{gO}$ vanish whenever
\eq
\ii(\varphi-\mathfrak{z}_1^{[j]})= 2\pi\ii N_\mathfrak{a} + H_\omega N_\mathfrak{b} \quad\text{or}\quad
\ii(\varphi-\mathfrak{z}_2^{[j]})= 2\pi\ii N_\mathfrak{a}+ H_\omega N_\mathfrak{b}
\label{eq:zero-conditions}
\endeq
and blow up whenever
\eq
\ii(\varphi-\mathfrak{p}_1^{[j]}) = 2\pi\ii N_\mathfrak{a} + H_\omega N_\mathfrak{b}\quad\text{or}\quad
\ii(\varphi-\mathfrak{p}_2^{[j]}) = 2\pi\ii N_\mathfrak{a}+ H_\omega N_\mathfrak{b},
\label{eq:pole-conditions}
\endeq
where $N_\mathfrak{a}$ and $N_\mathfrak{b}$ are arbitrary integers.
\begin{lemma}
For the approximations $\dot{U}^{[j]}_\mathrm{F}(\zeta;y_0)$, $j=1,3$, no two of the four conditions in \eqref{eq:zero-conditions}--\eqref{eq:pole-conditions} can hold simultaneously.  
\label{lem:DistinctDivisors}
\end{lemma}
\begin{proof}
As in the proof of Lemma~\ref{lem:SameDivisor} we use the Abel-Jacobi Theorem.  For $j=3$, we notice that taking $\ii$ times the differences of any two of the phase shifts in \eqref{eq:dotU-phases} gives a difference of Abel maps evaluated at two distinct points of $\mathcal{R}$.  The corresponding conditions will hold simultaneously if and only if there is a nonzero meromorphic function on $\mathcal{R}$ with a simple pole at one of these points and vanishing at the other point.  But having only one simple pole on $\mathcal{R}$, this function must be a constant, and hence to vanish anywhere it must vanish identically.  For $j=1$, taking into account that $z_0$ cannot equal any of the four roots $\alpha,\beta,\gamma,\delta$ so that $2a(z_0)$ can be written as a difference of Abel maps of distinct points of $\mathcal{R}$ over $z_0$, the same argument applies to any of the four differences $\ii (\mathfrak{z}_k^{[1]}-\mathfrak{p}_\ell^{[1]})$ for phase shifts defined in \eqref{eq:dotUtwist-phases}.  For the remaining two differences $\ii (\mathfrak{z}_1^{[1]}-\mathfrak{z}_2^{[1]}) = \nu a(0)- 2a(z_0) + a(\infty)$ and $\ii(\mathfrak{p}_1^{[1]}-\mathfrak{p}_2^{[1]})=\nu a(0)+2a(z_0)+a(\infty)$, we note that the proof of Lemma~\ref{lem:SameDivisor} showed that $2a(\infty)+2a(z_0) = 2\pi\ii N_\mathfrak{a}+H_\omega N_\mathfrak{b}$ for some integers $N_\mathfrak{a}$ and $N_\mathfrak{b}$.  Therefore, it suffices to prove that neither $\nu a(0)+3a(\infty)$ nor $\nu a(0)- a(\infty)$ can be an integer linear combination of $2\pi\ii$ and $H_\omega$.  For $\nu a(0)-a(\infty)$, the same argument as above applies.  For $\nu a(0)+3a(\infty)$ we move two of the factors of $a(\infty)$ ``onto the other sheet'' by using the Abel mapping on the Riemann surface $\mathcal{R}$ defined in the proof of Lemma~\ref{lem:SameDivisor} to write $\nu a(0)+3a(\infty)=\widetilde{a}(Q^\nu(0))+\widetilde{a}(Q^+(\infty))-2\widetilde{a}(Q^-(\infty))$.  Thus applying the Abel-Jacobi Theorem it is sufficient to prove that there can be no meromorphic function $k(Q)$ on $\mathcal{R}$ that does not vanish identically and has only one double pole at $Q^-(\infty)$ and simple zeros at $Q^\nu(0)$ and $Q^+(\infty)$.  Using the ``coordinate'' meromorphic functions $z(Q)$ and  $\widetilde{R}(Q)$ defined on $\mathcal{R}$ as in the proof of Lemma~\ref{lem:SameDivisor}, every nonzero function with a double pole at $Q^-(\infty)$ and at worst a simple pole at $Q^+(\infty)$ is necessarily a nonzero multiple of 
\eq
k(Q)=z(Q)^2-\widetilde{R}(Q) + c_1z(Q)+c_2
\endeq
for constants $c_1$ and $c_2$.  Since $z(Q^\nu(0))=0$ and $\widetilde{R}(Q^\nu(0))=\nu R(0)=4\nu s$, enforcing the condition $k(Q^\nu(0))=0$ requires taking $c_2=4\nu s=\pm 4$.  Now, the expansion $R(z)=z^2+2y_0z+4\kappa + \bo(z^{-1})$ as $z\to\infty$ implies also that $\widetilde{R}(Q)=z(Q)^2+2y_0z(Q) + 4\kappa + \bo(z(Q)^{-1})$ as $Q\to Q^+(\infty)$. Therefore, enforcing the condition that $k(Q)$ be analytic at $Q=Q^+(\infty)$ requires taking $c_1=2y_0$, and then demanding further that $k(Q^+(\infty))=0$ requires taking $c_2=4\kappa$.  But since $\kappa\neq\pm 1$, this is a contradiction with $c_2=4\nu s=\pm 4$.  Hence no such function $k(Q)$ exists, and the proof is finished.  An alternative indirect proof can be based on the differential equation \eqref{eq:elliptic-ODE} and the results of Section~\ref{sec:Elliptic-DiffEq}.
\end{proof}

Using the fact that $\mathrm{Re}(H_\omega)<0$ (implying that $2\pi\ii$ and $H_\omega$ are linearly independent over $\mathbb{R}$), we can solve for $N_\mathfrak{a}$ and $N_\mathfrak{b}$ and hence express these conditions as ``quantization rules''.  Thus, 
$\dot{U}^{[j]}_\mathrm{F}(\zeta;y_0)$ vanishes if and only if
\eq
\left(\mathfrak{Z}_{1,\mathfrak{a}}^{[j]}=N_\mathfrak{a}\in\mathbb{Z}\;\text{and}\;
\mathfrak{Z}_{1,\mathfrak{b}}^{[j]}=N_\mathfrak{b}\in\mathbb{Z}\right)\quad\text{or}\quad
\left(\mathfrak{Z}_{2,\mathfrak{a}}^{[j]}=N_\mathfrak{a}\in\mathbb{Z}\;\text{and}\;
\mathfrak{Z}_{2,\mathfrak{b}}^{[j]}=N_\mathfrak{b}\in\mathbb{Z}\right)
\label{eq:zero-conditions-2}
\endeq
and blows up if and only if
\eq
\left(\mathfrak{P}_{1,\mathfrak{a}}^{[j]}=N_\mathfrak{a}\in\mathbb{Z}\;\text{and}\;
\mathfrak{P}_{1,\mathfrak{b}}^{[j]}=N_\mathfrak{b}\in\mathbb{Z}\right)\quad\text{or}\quad
\left(\mathfrak{P}_{2,\mathfrak{a}}^{[j]}=N_\mathfrak{a}\in\mathbb{Z}\;\text{and}\;
\mathfrak{P}_{2,\mathfrak{b}}^{[j]}=N_\mathfrak{b}\in\mathbb{Z}\right)
\label{eq:pole-conditions-2}
\endeq
where
\eq
\begin{split}
\mathfrak{Z}_{k,\mathfrak{a}}^{[j]}:=\frac{\mathrm{Re}(H_\omega^*(\varphi-\mathfrak{z}_k^{[j]}))}{2\pi\mathrm{Re}(H_\omega)}\quad &\text{and}\quad
\mathfrak{Z}_{k,\mathfrak{b}}^{[j]}:=-\frac{\mathrm{Im}(\varphi-\mathfrak{z}_k^{[j]})}{\mathrm{Re}(H_\omega)}\\
\mathfrak{P}_{k,\mathfrak{a}}^{[j]}:=\frac{\mathrm{Re}(H_\omega^*(\varphi-\mathfrak{p}_k^{[j]}))}{2\pi\mathrm{Re}(H_\omega)}\quad &\text{and}\quad
\mathfrak{P}_{k,\mathfrak{b}}^{[j]}:=-\frac{\mathrm{Im}(\varphi-\mathfrak{p}_k^{[j]})}{\mathrm{Re}(H_\omega)}.
\end{split}
\endeq
Using \eqref{eq:F1-U} and expressing $c$ and $H_\omega$ in terms of the elliptic periods $Z_\mathfrak{a,b}$ by $c=Z_\mathfrak{a}$ and $H_\omega=2\pi\ii Z_\mathfrak{b}/Z_\mathfrak{a}$, these expressions are exactly the left-hand sides in \eqref{eq:intro-zeros-quantize}--\eqref{eq:intro-poles-quantize}.  

\subsection{Residues of the approximations at the Malgrange divisor}
\label{sec:MalgrangeResidues}
The goal of this section is explain what happens to the approximations of the rational functions near points in the parameter space where the approximations fail to exist.  On one hand, this will allow us to explain how B\"acklund transformations can be used to circumvent the non-existence issue.  On the other hand, we will then be able to identify zeros of a particular theta-function factor with approximations of zeros of the gH and gO polynomials themselves and thus prove Corollary~\ref{cor:polynomial-zeros}.  

The \emph{Malgrange divisor} of Riemann-Hilbert Problem~\ref{rhp:Pout} is the set of parameter values for which there is no solution of that problem.  By Lemma~\ref{lem:SameDivisor}, it is equivalently characterized by either of the conditions $\Theta(a(\infty)+a(z_0)+K\mp\ii\varphi)=0$.  According to the formul\ae\ \eqref{eq:U-dot-define} and \eqref{eq:U-dot-twist-final} and Lemma~\ref{lem:DistinctDivisors}, the Malgrange divisor gives rise to exactly one simple pole per period parallelogram in the $\zeta$-plane of each of the elliptic functions $\dot{U}_\mathrm{F}^{[3]}(\zeta;y_0)$ and $\dot{U}_\mathrm{F}^{[1]}(\zeta;y_0)$.  Each of these functions has one additional simple pole (of opposite residue) and two simple zeros in each parallelogram.  From the differential equations \eqref{eq:elliptic-ODE} and \eqref{eq:elliptic-ODE-twist} satisfied by $\dot{U}_\mathrm{F}^{[3]}(\zeta;y_0)$ and $\dot{U}_\mathrm{F}^{[1]}(\zeta;y_0)$ respectively, it is easy to see that each simple pole has residue $\pm 1$.  

\subsubsection{Malgrange residues of $\dot{U}_\mathrm{F}^{[3]}$}
We now calculate the sign of the residue for the pole of $\dot{U}_\mathrm{F}^{[3]}(\zeta;y_0)$ at the Malgrange divisor.  We do this by applying a homotopy argument:  the desired residue is a continuous function of the parameters in the formula for $\dot{U}_\mathrm{F}^{[3]}(\zeta;y_0)$ such as $y_0$, provided the spectral curve remains nondegenerate, i.e., of class $\{1,1,1,1\}$.  In fact, the formula for $\dot{U}_\mathrm{F}^{[3]}(\zeta;y_0)$ depends on these parameters through (i) the phase $\xi$ (see \eqref{eq:F1-U}) and (ii) the distinct roots $\alpha$, $\beta$, $\gamma$, and $\delta$.  If the latter parameters are given instead, then $\dot{U}_\mathrm{F}^{[3]}(\zeta;y_0)$ is determined as an elliptic function of $\zeta$.  In other words, neither the explicit formula for $\dot{U}_\mathrm{F}^{[3]}(\zeta;y_0)$ nor the argument that it satisfies the differential equation \eqref{eq:elliptic-ODE} with $P(z)$ taken in the form $P(z)=(z-\alpha)(z-\beta)(z-\gamma)(z-\delta)$ requires the specific relation between $E$, $y_0$, and $\kappa$ determined by the Boutroux conditions \eqref{eq:Boutroux}.  Indeed, the latter relation is needed only to control the approximation of the matrix $\mathbf{O}(z)$ by its parametrix $\dot{\mathbf{O}}(z)$.  Therefore, we will write $\dot{U}_\mathrm{F}^{[3]}(\zeta)$ for $\dot{U}_\mathrm{F}^{[3]}(\zeta;y_0)$ and compute the desired residue as a continuous function of the parameters $\xi$, $\alpha$, $\beta$, $\gamma$, and $\delta$ by a suitable (generally artificial) homotopy that need not be consistent with varying $y_0$ and solving \eqref{eq:Boutroux} for $E$.  That said, since the residue is either $1$ or $-1$, it will remain constant along such a homotopy.  The homotopy we select is to fix $\xi$ and to deform the given roots $(\alpha,\beta,\gamma,\delta)$ to the points $(2\ee^{\frac{\ii\pi}{4}},2\ee^{\frac{3\ii\pi}{4}},2\ee^{\frac{5\ii\pi}{4}},2\ee^{\frac{7\ii\pi}{4}})$ without allowing any intermediate degeneration.  (Note that the target configuration is the actual root configuration for $y_0=0$ and $\kappa=0$, which also yields $E=0$; hence it is reachable by homotopy in $y_0$ if initially $y_0\in\rectangle(\kappa)$.  It is not clear whether it can be reached by such a homotopy if initially $y_0\in\TR(\kappa)\cup\TI(\kappa)$, but it is not necessary either.)  In the final configuration we have, from \eqref{eq:z0-define}, that $z_0=\infty$.  By expanding the Abel map $a(z_0)$ for large $z_0$, one sees that as the target configuration is approached in the parameter space,
\eq
z_0\Theta(a(\infty)-a(z_0)-K)\to -\frac{2\pi\ii}{c}\Theta'(K)
\endeq
where the constants $c$ and $K$ are determined from the target configuration.  Therefore, from \eqref{eq:U-dot-define} and \eqref{eq:psi-simpler} we have, in the target configuration
\eq
\dot{U}^{[3]}_\mathrm{F}(\zeta)=-\frac{2\pi\ii}{c}\frac{\Theta(2a(\infty)+K)\Theta'(K)}{\Theta(a(0)+a(\infty)+K)\Theta(a(0)-a(\infty)-K)}\cdot\frac{\Theta(a(0)+a(\infty)+K-\ii\varphi)\Theta(a(0)-a(\infty)-K+\ii\varphi)}{\Theta(2a(\infty)+K-\ii\varphi)\Theta(\ii\varphi-K)}.
\endeq
Now we compute the residue at a zero of the factor $\Theta(2a(\infty)+K-\ii\varphi)$.  Taking into account \eqref{eq:F1-U} for the $\zeta$-dependence of $\varphi$, and evaluating the residue at $\ii\varphi=2a(\infty)$ gives
\eq
\mathop{\mathrm{Res}}_{\zeta:\ii\varphi=2a(\infty)}\dot{U}_\mathrm{F}^{[3]}(\zeta)=-\frac{\Theta(2a(\infty)+K)}{\Theta(a(0)+a(\infty)+K)\Theta(a(0)-a(\infty)-K)}\cdot
\frac{\Theta(a(0)-a(\infty)+K)\Theta(a(0)+a(\infty)-K)}{\Theta(2a(\infty)-K)}.
\endeq
Finally, using \eqref{eq:theta-identities} and $K=-K + 2\pi\ii + H_\omega$ we get
\eq
\mathop{\mathrm{Res}}_{\zeta:\ii\varphi=2a(\infty)}\dot{U}_\mathrm{F}^{[3]}(\zeta)=-\ee^{-\frac{1}{2}H_\omega}\ee^{K} = 1
\endeq
in the target configuration.  Applying the homotopy argument then shows that the residue is also $+1$ in the arbitrary initial configuration.

\subsubsection{Malgrange residues of $\dot{U}_\mathrm{F}^{[1]}$ and $\dot{U}_\mathrm{F}^{[2]}$}
With this result established, we now use the differential relation \eqref{eq:dotU-dotUtwist} where $\dot{U}(\zeta)=\dot{U}^{[3]}_\mathrm{F}(\zeta)$ and $\dot{U}_\tw(\zeta_\tw)=\dot{U}^{[1]}_\mathrm{F}(\zeta_\tw)$ with $\zeta_\tw=\sqrt{\tfrac{1}{2}(1-s\kappa)}\zeta$ to find that at the simple poles corresponding to the Malgrange divisor,
\eq
\mathop{\mathrm{Res}}\dot{U}^{[1]}_\mathrm{F}(\zeta;y_0)=-1.
\endeq
Finally, it follows directly that the residue of $\dot{U}^{[2]}_\mathrm{F}(\zeta;y_0)=\ii\dot{U}^{[1]}_\mathrm{F}(-\ii\zeta;-\ii y_0)$ at the Malgrange divisor is
\eq
\mathop{\mathrm{Res}}\dot{U}^{[2]}_\mathrm{F}(\zeta;y_0)=1.
\endeq

\subsubsection{Removal of the condition $(y_0,\zeta,T)\in\mathcal{S}(\epsilon)$}
The final step in proving Theorems~\ref{thm:Hermite-elliptic} and \ref{thm:Okamoto-elliptic} is to remove the condition that $(y_0,\zeta,T)\in\mathcal{S}(\epsilon)$, which bounds $\zeta$ away from the Malgrange divisor.  For each $(y_0,T)$, the Malgrange divisor is a uniform lattice $\Lambda_\mathrm{M}$ in the $\zeta$-plane with lattice vectors determined from $y_0$ and $\kappa$ and an offset involving $T$.  The lattice $\Lambda_\mathrm{M}$ consists of all poles of the elliptic function $\dot{U}^{[1]}_\mathrm{F}(\zeta)$ of residue $-1$ (equivalently all poles of the elliptic function $\dot{U}^{[3]}_\mathrm{F}(\zeta)$ of residue $1$).  The poles of $\dot{U}^{[1]}_\mathrm{F}(\zeta)$ of residue $1$ form a lattice $\Lambda^{[1]}_+$ congruent to $\Lambda_\mathrm{M}$ but, by Lemma~\ref{lem:DistinctDivisors}, having a different offset.  Likewise, the poles of $\dot{U}^{[3]}_\mathrm{F}(\zeta)$ of residue $-1$ form another congruent lattice $\Lambda^{[3]}_-$ disjoint from $\Lambda_\mathrm{M}$.  Let $\delta>0$ be sufficiently small given $\epsilon>0$ that a $\delta$-neighborhood of $\Lambda^{[1]}_+$ and $\Lambda^{[3]}_-$ is contained within $\mathcal{S}(\epsilon)$.  Then $u^{[1]}_\mathrm{F}(T^\frac{1}{2}y_0+T^{-\frac{1}{2}}\zeta;m,n)^{-1}=T^{-\frac{1}{2}}(\dot{U}^{[1]}_\mathrm{F}(\zeta)^{-1}+\bo(T^{-1}))$ holds uniformly for bounded $\zeta$ in a $\delta$-neighborhood of $\Lambda^{[1]}_+$ and $u^{[3]}_\mathrm{F}(T^\frac{1}{2}y_0+T^{-\frac{1}{2}}\zeta;m,n)^{-1}=T^{-\frac{1}{2}}(\dot{U}^{[3]}_\mathrm{F}(\zeta)^{-1}+\bo(T^{-1}))$ holds uniformly for bounded $\zeta$ in a $\delta$-neighborhood of $\Lambda^{[3]}_-$.  Applying a standard perturbation argument based on the Analytic Implicit Function Theorem, one sees that each point of $\Lambda^{[1]}_+$ attracts exactly one simple pole of $u^{[1]}_\mathrm{F}(T^\frac{1}{2}y_0+T^{-\frac{1}{2}}\zeta;m,n)$ and each point of $\Lambda^{[3]}_-$ attracts one simple pole of $u^{[3]}_\mathrm{F}(T^\frac{1}{2}y_0+T^{-\frac{1}{2}}\zeta;m,n)$, of positive and negative residue, respectively.  The attracted poles lie within a distance of $\bo(T^{-1})$ from the corresponding attracting lattice points in any given bounded region of the $\zeta$-plane.  Every pole of the opposite residue is necessarily attracted to a point of the Malgrange divisor lattice $\Lambda_\mathrm{M}$, but since the error terms cannot be controlled near this lattice we cannot say for sure whether there might be clusters of additional poles attracted to these lattice points as well.  We may calculate the winding number index of the rational solution about a circle of radius $\epsilon$ centered at a point of $\Lambda_\mathrm{M}$, but this only shows that any excess poles must be paired with an equal number of zeros;  computing the integral of the rational solution around this circle and applying the Residue Theorem shows that further there must be an equal number of excess poles of opposite residues.  

Recall that the Boiti-Pempinelli symmetry $\mathcal{S}_\updownarrow$ discussed in Section~\ref{sec:intro-Baecklund} yields the identity $u^{[3]}_\mathrm{F}(x;m,n)=\ii u^{[3]}_\mathrm{F}(-\ii x;n,m)$ (cf.\@ \eqref{eq:symmetry-1-2}).  It follows easily that poles of residue $\pm 1$ of $u^{[3]}_\mathrm{F}(\cdot;m,n)$ correspond under rotation of the argument by $\tfrac{1}{2}\pi$ to poles of residue $\mp 1$ of $u^{[3]}_\mathrm{F}(\cdot;n,m)$.  The poles of $u^{[3]}(\cdot;m,n)$ of residue $-1$ whose isolated images in the $\zeta$-plane lie close to the lattice $\Lambda^{[3]}_-$ are now poles of residue $1$ of $u^{[3]}_\mathrm{F}(\cdot;n,m)$.  Being as those poles are isolated in the (rotated) $\zeta$-plane, a residue integral calculation shows that each one must lie close to a lattice point at which the elliptic function approximation of $u^{[3]}_\mathrm{F}(\cdot;n,m)$ has a pole of residue $1$.  That approximation can be obtained by a different case of the asymptotic analysis of $u^{[3]}_\mathrm{F}(\cdot;m,n)$ in which the indices $(m,n)$ are permuted, corresponding in the limit to a change of sign of $\kappa_\infty$.  But in the approximation of any rational solution of type $3$, poles of the approximating function of residue $1$ form the Malgrange divisor of the corresponding instance of Riemann-Hilbert Problem~\ref{rhp:Pout}.  This is an indirect proof that in fact each such pole attracts exactly one pole of the same residue, even though the error has not been controlled directly near the Malgrange divisor.  However, to provide this control is now easy, since $u^{[3]}_\mathrm{F}(T^\frac{1}{2}y_0+T^{-\frac{1}{2}}\zeta;m,n)^{-1}$ and its elliptic function approximation are both analytic in an $\epsilon$-neighborhood of $\Lambda_\mathrm{M}$.  Indeed, it follows that the error term is also analytic on this neighborhood, so applying the maximum modulus principle on each $\epsilon$-disk centered at a point of $\Lambda_\mathrm{M}$ proves that the error term is uniformly small on the disk since it is small on the boundary by the primary estimate valid on $\mathcal{S}(\epsilon)$.  

To deal with the function $u^{[1]}_\mathrm{F}(x;m,n)$, we compose $\mathcal{S}_\updownarrow$ with the symmetry $\mathcal{S}_\tw$ also discussed in Section~\ref{sec:intro-Baecklund}, which maps rational solutions of type $1$ to solutions of type $3$ (cf.\@ \eqref{eq:symmetry-1-from-3}).    It is easy to check that the composite transformation $\mathcal{S}_\tw\circ\mathcal{S}_\updownarrow\circ\mathcal{S}_\tw^{-1}$ maps $u^{[1]}_\mathrm{F}(x;m,n)$ to another function of type $1$ with different (large) indices, and again the transformation has the effect of rotation in the complex $x$-plane that swaps the signs of the residues without introducing any new poles or removing any preexisting ones.  So the same argument applies again to show that $u^{[1]}_\mathrm{F}(T^\frac{1}{2}y_0+T^{-\frac{1}{2}}\zeta;m,n)^{-1}=T^{-\frac{1}{2}}(\dot{U}^{[1]}_\mathrm{F}(\zeta)+\bo(T^{-1}))$ holds uniformly near all poles for $\zeta$ bounded, whether they lie near the Malgrange divisor or not.  

At last this completes the proof of Theorems~\ref{thm:Hermite-elliptic} and \ref{thm:Okamoto-elliptic}.

\subsubsection{Accurate approximation of poles and zeros of rational Painlev\'e-IV solutions.  Zeros of the gH and gO polynomials.}
We are now also in a position to give the proof of Corollary~\ref{cor:poles-and-zeros}.
\begin{proof}[Proof of Corollary~\ref{cor:poles-and-zeros}]
Setting $\zeta=0$, let $y_0$ be a value satisfying one of the conditions in \eqref{eq:zero-conditions-2} or \eqref{eq:pole-conditions-2}.  Fixing this value of $y_0$, and letting $\zeta$ be free, either $f(\zeta-\zeta_0)$ in the former case or $f(\zeta-\zeta_0)^{-1}$ in the latter case has a simple zero at $\zeta=0$.  Putting an artificial coefficient $\delta\in [0,1]$ on the $\bo(T^{-1})$ perturbing term in Theorem~\ref{thm:Hermite-elliptic} or \ref{thm:Okamoto-elliptic}, it then follows from the Analytic Implicit Function Theorem that there is a unique simple zero of either $u(x)$ or its reciprocal that depends on $\delta$ and persists up to $\delta=1$ when $(m,n)$ are sufficiently large and hence $T^{-1}$ is sufficiently small.  As an analytic function of $\delta$, this zero obviously satisfies $\zeta=\bo(T^{-1})$, which under $x=T^\frac{1}{2}y_0+T^{-\frac{1}{2}}\zeta$ can also be interpreted as a perturbation of $y_0$ of order $\bo(T^{-2})$ for $\zeta=0$.

%On the other hand, if $y_0$ is a point in the Malgrange divisor for $\zeta=0$, then while $u(x)$ is a rational function we do not have control of the $O(T^{-1})$ error term in Theorem~\ref{thm:Hermite-elliptic} or \ref{thm:Okamoto-elliptic} for $\zeta\neq 0$ but small, so we cannot apply the above approach.  However, we can compute the winding number of $u(x)^{-1}$ about a circle of arbitrarily small but fixed radius in the $\zeta$-plane centered at $\zeta=0$, along which the $O(T^{-1})$ estimate of the error term is uniformly accurate.  Hence the winding number equals that of the leading term, which is $1$ for $f(\zeta-\zeta_0)^{-1}$.  Therefore $u(x)^{-1}$ also has winding number $1$ and we conclude that there is one simple zero and possibly an equal number of additional poles and zeros of $u(x)^{-1}$ within the image of the circle in the $x$ or $y_0$-plane.  

Note that we actually cannot fix a value $y_0$, because the possible values of $y_0$ for $\zeta=0$ depend strongly on $(m,n)$.  But we can always select a sequence of values of $y_0$ depending on $(m,n)$ large and lying within the fixed compact set $C$ and perform the computation separately for each $(m,n)$.  Since by Theorems~\ref{thm:Hermite-elliptic} and \ref{thm:Okamoto-elliptic} the size of the $\bo(T^{-1})$ error term is uniform given $C$, the proof is complete.
\end{proof}

The proof of Corollary~\ref{cor:polynomial-zeros} then follows easily.
\begin{proof}[Proof of Corollary~\ref{cor:polynomial-zeros}]
According to the logarithmic derivative formul\ae\ \eqref{eq:gH-tau} and \eqref{eq:gO-tau} and the expressions for the type-$1$ tau functions $\tau^{[1]}_\mathrm{F}(x;m,n)$ given for families $\mathrm{F}=\mathrm{gH}$ and $\mathrm{F}=\mathrm{gO}$ in Tables~\ref{tab:gH} and \ref{tab:gO} respectively, the roots of the polynomials $H_{m,n}(x)$ and $Q_{m,n}(x)$ are precisely the poles of residue $-1$ of $u_\mathrm{F}^{[1]}(x;m,n)$.  According to Corollary~\ref{cor:poles-and-zeros}, these are approximated by the points in the Malgrange divisor for Riemann-Hilbert Problem~\ref{rhp:Pout}, which one can check satisfy \eqref{eq:pole-conditions-2} for $k=2$.
\end{proof}

%% file: OkamotoAppendices.tex
% For a first appendix, we put in the erstwhile section 8 on isomonodromy theory.  
\input{IsomonodromyTheory}

\section{Symmetries of branch points of equilibria:  Proof of Proposition~\ref{prop:triangles}}
\label{app:Equilateral}
Recall that Proposition~\ref{prop:triangles} asserts that the three roots of the equation $B(y_0;\kappa)=0$ (see \eqref{eq:branch-points}) lying in any coordinate half-plane ($\pm\mathrm{Re}(y_0)>0$ or $\pm\mathrm{Im}(y_0)>0$) are the vertices of an equilateral triangle.  We now prove this proposition.
\begin{proof}
Any configuration of eight points symmetric with respect to the real and imaginary axes and consisting of the vertices of two opposite equilateral triangles in the open right and left half-planes together with a conjugate pair of purely imaginary points must be the roots of a polynomial $b(y_0;\alpha,\beta,\gamma)$ of the form
\eq
b(y_0;\alpha,\beta,\gamma):=((y_0-\alpha)^3-\beta^3)((y_0+\alpha)^3+\beta^3)(y_0^2+\gamma^2)
\label{eq:by0}
\endeq
for real parameters $\alpha$ (the centers of the triangles of vertices in the open right and left half-planes are $y_0=\pm\alpha$), $\beta$ (the distance from the center of each triangle to any of its vertices is $|\beta|$), and $\gamma$ (the purely imaginary conjugate pair of roots is $y_0=\pm\ii\gamma$).  Equating the coefficients of powers of $y_0$ between $B(y_0;\kappa)$ given by \eqref{eq:branch-points} and $b(y_0;\alpha,\beta,\gamma)$ given by \eqref{eq:by0},
we see that $B(y_0;\kappa)$ can be written in the form $b(y_0;\alpha,\beta,\gamma)$ provided that $\gamma^2=3\alpha^2$ (matching the coefficients of $y_0^6$) and that after eliminating $\gamma^2$,
\eq
\alpha^4+\alpha\beta^3=4(\kappa^2+3) \quad\text{(matching the coefficients of $y_0^4$),}
\label{eq:y0fourth}
\endeq
\eq
8\alpha^6-20\alpha^3\beta^3-\beta^6=-64\kappa(\kappa^2-9)\quad\text{(matching the coefficients of $y_0^2$), and}
\label{eq:y0second}
\endeq
\eq
\alpha^8+2\alpha^5\beta^3+\alpha^2\beta^6=16(\kappa^2+3)^2 \quad\text{(matching the constant terms).}  
\label{eq:y0zeroth}
\endeq
Obviously \eqref{eq:y0fourth} implies \eqref{eq:y0zeroth}, so there are only two conditions:  \eqref{eq:y0fourth} and \eqref{eq:y0second}, which amount to two equations on the two remaining unknowns $\alpha$ and $\beta$.  
We can eliminate $\beta^3$ explicitly using \eqref{eq:y0fourth}:
%Combining \eqref{eq:y0second}--\eqref{eq:y0zeroth} allows $\beta^3$ to be explicitly eliminated via
\eq
%\beta^3=\frac{1}{18\alpha^5}\left(9\alpha^8+64\kappa(\kappa^2-9)\alpha^2-16(\kappa^2+3)^2\right).
\beta^3=\frac{4(\kappa^2+3)-\alpha^4}{\alpha}.
\label{eq:beta-alpha-c}
\endeq
Using this in 
%\eqref{eq:y0fourth} 
\eqref{eq:y0second}
one arrives at an 8th-degree polynomial equation for $\alpha$.  Comparing with \eqref{eq:branch-points}, it is easy to see that the equation on $\alpha$ is exactly $B(\sqrt{3}\ii\alpha;\kappa)=0$.  For all $\kappa\in\mathbb{R}\setminus\{-1,1\}$ we can therefore determine a unique positive solution $\alpha=\alpha(\kappa)>0$ that corresponds to the unique positive imaginary root of $B(\cdot;\kappa)$.  Therefore, with $\beta^3$ determined from $\alpha(\kappa)>0$ by \eqref{eq:beta-alpha-c} and with $\gamma^2=3\alpha(\kappa)^2$, we have the identity $B(y_0;\kappa)=b(y_0;\alpha,\beta,\gamma)$ which proves that $B(\cdot;\kappa)$ has two opposite triads of roots forming the vertices of equilateral triangles with centers $\pm\alpha(\kappa)\neq 0$, along with the purely imaginary pair $y_0=\pm\ii\gamma=\pm \sqrt{3}\ii\alpha(\kappa)$.  

We next check that $\beta>0$ (which ensures that the real vertex of each triangle is further from the origin than the center) and that $\beta<2\alpha$ (which ensures that all three vertices of each triangle lie in the same right or left half-plane as the center).  But these inequalities hold for $\kappa=1\pm\epsilon$ and $\epsilon>0$ sufficiently small by the perturbation theory described above, which in particular localizes the triangles near $y_0=\pm 2$.  Putting $\beta=0$ into \eqref{eq:y0fourth}--\eqref{eq:y0second} and eliminating $\alpha$ yields $\kappa=\pm 1$, so $\beta>0$ for $\kappa=1\pm\epsilon$ 
with $\epsilon>0$ small implies that $\beta>0$ for all $\kappa\in (-1,1)$ and all $\kappa>1$.  Likewise, putting $\beta=2\alpha$ into \eqref{eq:y0fourth}--\eqref{eq:y0second} and eliminating $\alpha$ yields again $\kappa=\pm 1$, so $\beta<2\alpha$ for $\kappa=1\pm\epsilon$ with $\epsilon>0$ small implies that $\beta<2\alpha$ holds for all $\kappa\in (-1,1)$ and all $\kappa>1$.  This finally shows that if $\kappa\in (-1,1)$ or $\kappa>1$, the roots of $B(y_0;\kappa)$ in the open right (left) half-plane form the vertices of an equilateral triangle symmetric with respect to reflection in the real axis and with its real vertex lying to the right (left) of its center.  By $B(y_0;\kappa)=B(\ii y_0;-\kappa)$ a similar statement governs the roots of $B(y_0;\kappa)$ in the open upper/lower half-planes for $\kappa<-1$.

For $\kappa\in (-1,1)$ or $\kappa> 1$, the remaining two roots form a purely imaginary pair $y_0=\pm \sqrt{3}\ii\alpha(\kappa)$ following from 
the identity $\gamma^2=3\alpha^2$.  Some simple trigonometry illustrated in Figure~\ref{fig:OkamotoStar}
\begin{figure}[h]
\begin{center}
\includegraphics[width=4 in]{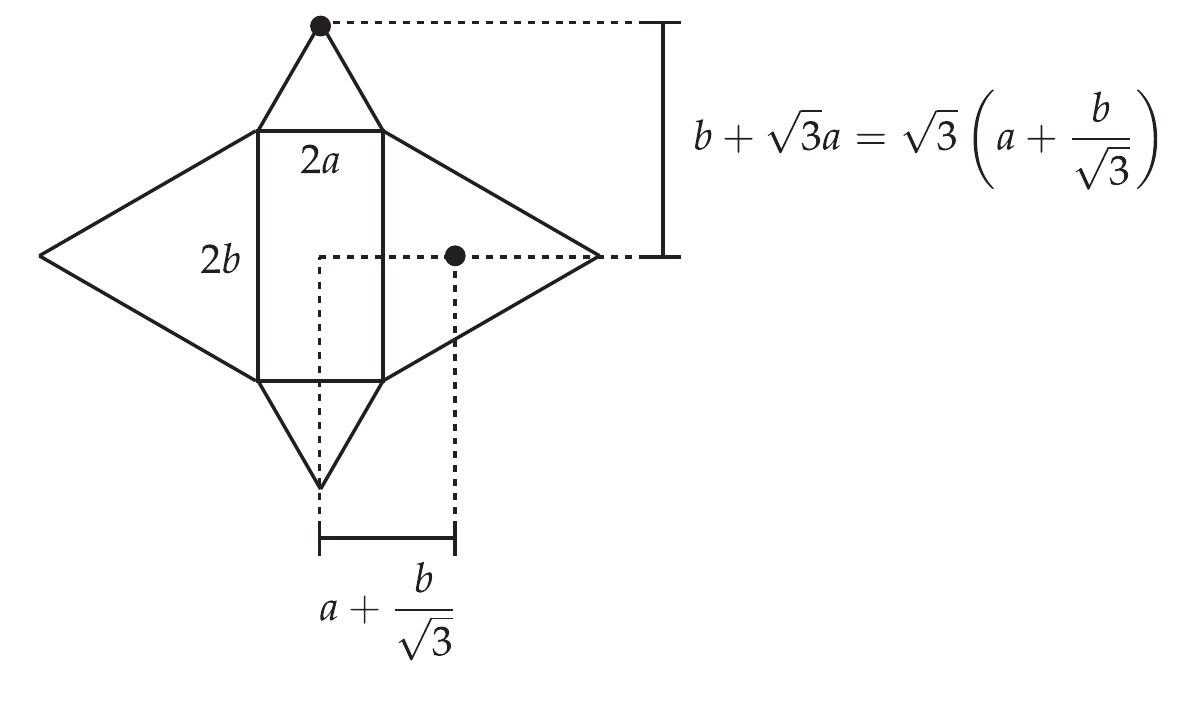}
\end{center}
\caption{For any configuration of equilateral triangles attached to the edges of a rectangle centered at the origin, the distance of the extremal vertex of any triangle from the origin is proportional by $\sqrt{3}$ to the distance of the center of either neighboring triangle to the origin.}
\label{fig:OkamotoStar}
\end{figure}
then shows that the triads of roots in the open upper and lower half-planes also form the vertices of opposite equilateral triangles with their imaginary vertices further from the origin than their centers.  (Alternatively, the whole argument of equating $B(y_0;\kappa)$ with $b(y_0;\alpha,\beta,\gamma)$ can be repeated replacing $b(y_0;\alpha,\beta,\gamma)$ with a polynomial of the form $((y_0-\ii\alpha)+\ii\beta^3)((y_0+\ii\alpha)-\ii\beta^3)(y_0^2-\gamma^2)$ for real $\alpha$, $\beta$, and $\gamma$, modeling a pair of opposite real roots and two opposite equilateral triangles of roots in the upper and lower half-planes.)  For $\kappa<-1$ the rotation symmetry $B(y_0;\kappa)=B(\ii y_0;-\kappa)$ shows that the triads of roots in the right/left half-planes form the vertices of equilateral triangles.
\end{proof}

\clearpage
\section{Diagrams and tables for steepest-descent analysis on Boutroux domains}
\label{app:DiagramsAndTables}
Here we gather $z$-plane diagrams and tables referred to in Section~\ref{sec:G1}.  In all plots gray shading means $\mathrm{Re}(h(z))<0$ and white background means $\mathrm{Re}(h(z))>0$.
\subsection{The gO case with $y_0\in\rectangle(\kappa)$ and $s=1$}
Here we use representative values of $y_0=0$ and $\kappa=0$.  
\begin{figure}[h!]
\begin{subfigure}{\textwidth}
\begin{center}
\includegraphics[width=5in]{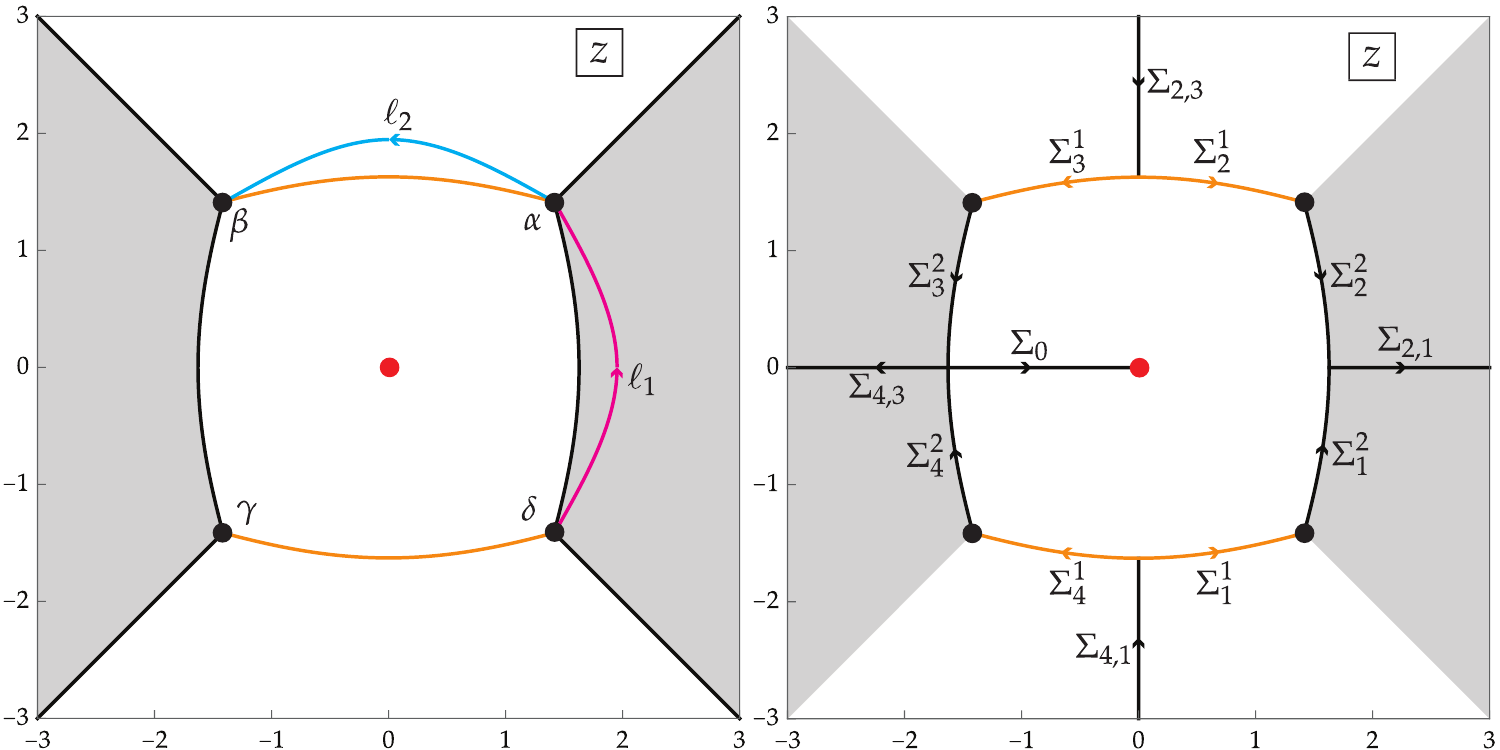}
\end{center}
\subcaption{Left:
Stokes graph including branch cuts (orange) for $h'(z)$ and contours $\ell_1$ and $\ell_2$ appearing in \eqref{eq:BoutrouxConstants}.  Right:
arcs of the jump contour $\Sigma$ for $\mathbf{M}(z)$; note that the arcs $\Sigma_j^k$, $j=1,\dots,4$ and $k=1,2$, lie on the Stokes graph. }
\label{fig:SignsContourLenses-y0-kappa0-s1}
\end{subfigure}
%Top right:  
%Right:
%arcs of the jump contour $\Sigma$ for $\mathbf{M}(z)$; note that the arcs $\Sigma_j^k$, $j=1,\dots,4$ and $k=1,2$, lie on the Stokes graph.  %Below: the jump matrix $\mathbf{W}$ such that $\widetilde{\mathbf{O}}
%\end{center}
%\caption{%
%For $y_0=0$, $\kappa=0$ (representative of $y_0\in\rectangle(\kappa)$, $\kappa\in (-1,1)$), and $s=1$; gO case.  Top left:  
%Left:
%Stokes graph including branch cuts (orange) for $h'(z)$ and contours $\ell_1$ and $\ell_2$ appearing in \eqref{eq:BoutrouxConstants}.  
%Top right:  
%Right:
%arcs of the jump contour $\Sigma$ for $\mathbf{M}(z)$; note that the arcs $\Sigma_j^k$, $j=1,\dots,4$ and $k=1,2$, lie on the Stokes graph.  %Below: the jump matrix $\mathbf{W}$ such that $\widetilde{\mathbf{O}}_+(z)=\widetilde{\mathbf{O}}_-(z)\mathbf{W}$.  
%In all plots: gray shading means $\mathrm{Re}(h(z))<0$ and white background means $\mathrm{Re}(h(z))>0$.
%}
%\end{figure}
%\begin{figure}[h!]
%\begin{center}
\medskip
\begin{subfigure}{\textwidth}
\begin{center}
\includegraphics[width=4 in]{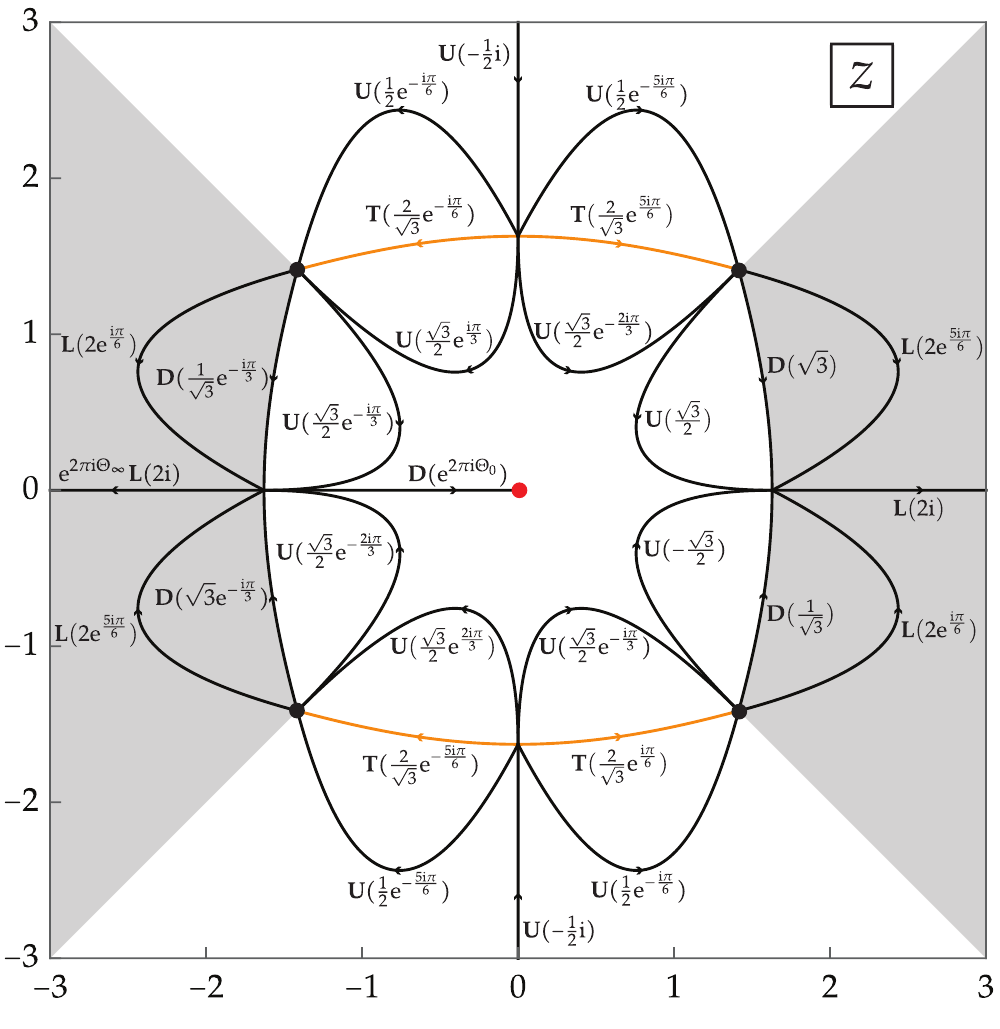}
\end{center}
\subcaption{The jump matrix $\mathbf{W}$ such that $\widetilde{\mathbf{O}}_+(z)=\widetilde{\mathbf{O}}_-(z)\mathbf{W}$.}
\label{fig:JustLenses-y0-kappa0-s1}
\end{subfigure}
%\caption{The jump matrix $\mathbf{W}$ such that $\widetilde{\mathbf{O}}_+(z)=\widetilde{\mathbf{O}}_-(z)\mathbf{W}$.}
%\label{fig:JustLenses-y0-kappa0-s1}
\caption{Diagrams for the gO case with $y_0\in\rectangle(\kappa)$ and $s=1$.}
\end{figure}
\begin{figure}[h]\ContinuedFloat
\begin{subfigure}{\textwidth}
\begin{center}
\includegraphics[width=5 in]{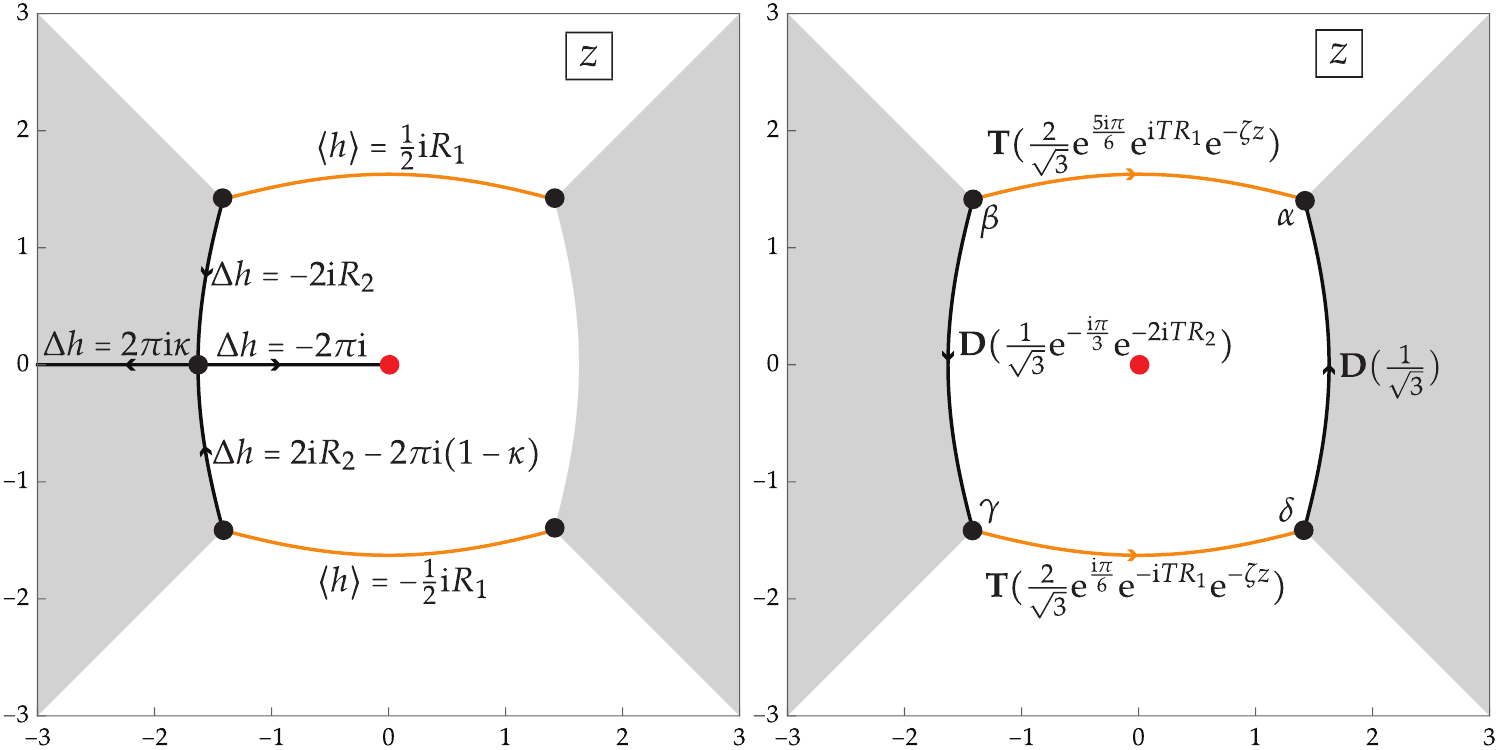}
\end{center}
\subcaption{Left:  the jump conditions satisfied by $h(z)$ with a suitable choice of integration constant.  Right:  the resulting jump conditions for the outer parametrix $\dot{\mathbf{O}}^{\mathrm{out}}(z)$.}
\label{fig:hJAO-y0-kappa0-s1}
\end{subfigure}
\caption{Diagrams for the gO case with $y_0\in\rectangle(\kappa)$ and $s=1$ (contd.)}
\end{figure}
\begin{table}
\caption{Inner parametrix data for the gO case with $y_0\in\rectangle(\kappa)$ and $s=1$.}
\renewcommand{\arraystretch}{0.45}
\begin{center}
\begin{tabular}{@{}|l|l|c|c|c|c|@{}}
\hline
\multirow{2}{*}{$p$}&\multirow{2}{*}{Conformal map $W:D_p\to\mathbb{C}$}&\multicolumn{2}{c|}{Ray Preimages in $D_p$} & \multicolumn{2}{c|}{$\mathbf{C}(z)$ in $D_p$}\\
\cline{3-6}
& & $\arg(W)$ & Preimage & Value $\mathbf{C}$ & Subdomain of $D_p$ \\
\hline\hline
\multirow{4}{*}{$\alpha$} & \multirow{4}{*}{\makecell[l]{$(2 h(z)-2 h(\alpha))^{2/3}$,\\ 
continued from $\Sigma_2^{1+}$%; $h(\alpha)$\\ defined by limit along $\Sigma_2^{1+}$
}} &
\shortstrut $0$ & $\Sigma_2^{1+}$ & 
\multirow{2}{*}{$\mathbf{D}(\sqrt{\tfrac{3}{2}}\ee^{\frac{5\ii\pi}{6}})$} & 
\multirow{2}{*}{$D_\alpha\cap\circledomain$} \\
\cline{3-4}
&& \shortstrut $\tfrac{2}{3}\pi$ & $\Sigma_2^{1-}$ \& $\Sigma_2^{2-}$ && \\
\cline{3-4}\cline{5-6}
&& \shortstrut $-\tfrac{2}{3}\pi$ & $\Sigma_2^{2+}$ & 
\multirow{2}{*}{$\mathbf{T}(\sqrt{2}\ee^{\frac{2\ii\pi}{3}})$} & 
\multirow{2}{*}{$D_\alpha\setminus\circledomain$} \\
\cline{3-4}
&& \shortstrut $\pm\pi$ & $\Sigma_2^2$ && \\
\hline
\hline
\multirow{4}{*}{$\beta$} & \multirow{4}{*}{\makecell[l]{$(2h(z)-2h(\beta))^{2/3}$,\\
continued from $\Sigma_3^{1-}$; $h(\beta)$\\
defined by limit along $\Sigma_3^{1-}$}} & \shortstrut $0$ & $\Sigma_3^{1-}$ &
\multirow{2}{*}{$\mathbf{D}(\sqrt{\tfrac{3}{2}}\ii)$} & 
\multirow{2}{*}{$D_\beta\cap\circledomain$} \\
\cline{3-4}
&& \shortstrut $\tfrac{2}{3}\pi$ & $\Sigma_3^{2-}$ && \\
\cline{3-4}\cline{5-6}
&& \shortstrut $-\tfrac{2}{3}\pi$ & $\Sigma_3^{1+}$ \& $\Sigma_3^{2+}$ & 
\multirow{2}{*}{$\mathbf{T}(\sqrt{2}\ee^{\frac{\ii\pi}{3}})$} & 
\multirow{2}{*}{$D_\beta\setminus\circledomain$} \\
\cline{3-4}
&& \shortstrut $\pm\pi$ & $\Sigma_3^2$ && \\
\hline
\hline
\multirow{4}{*}{$\gamma$} & \multirow{4}{*}{\makecell[l]{$(2h(z)-2h(\gamma))^{2/3}$,\\
continued from $\Sigma_4^{1+}$; $h(\gamma)$ \\
defined by limit along $\Sigma_4^{1+}$}} & \shortstrut $0$ & $\Sigma_4^{1+}$ &
\multirow{2}{*}{$\mathbf{D}(\sqrt{\tfrac{3}{2}}\ii)$} & 
\multirow{2}{*}{$D_\gamma\cap\circledomain$} \\
\cline{3-4}
&& \shortstrut $\tfrac{2}{3}\pi$ & $\Sigma_4^{1-}$ \& $\Sigma_4^{2-}$ && \\
\cline{3-4}\cline{5-6}
&& \shortstrut $-\tfrac{2}{3}\pi$ & $\Sigma_4^{2+}$ & 
\multirow{2}{*}{$\mathbf{T}(\sqrt{2}\ee^{\frac{2\ii\pi}{3}})$} & 
\multirow{2}{*}{$D_\gamma\setminus\circledomain$} \\
\cline{3-4}
&& \shortstrut $\pm\pi$ & $\Sigma_4^2$ && \\
\hline
\hline
\multirow{4}{*}{$\delta$} & \multirow{4}{*}{\makecell[l]{$(2h(z)-2h(\delta))^{2/3}$,\\
continued from $\Sigma_1^{1-}$ %; $h(\delta)$ \\defined by limit along $\Sigma_1^{1-}$
}} & \shortstrut $0$ & $\Sigma_1^{1-}$ &
\multirow{2}{*}{$\mathbf{D}(\sqrt{\tfrac{3}{2}}\ee^{\frac{\ii\pi}{6}})$} & 
\multirow{2}{*}{$D_\delta\cap\circledomain$} \\
\cline{3-4}
&& \shortstrut $\tfrac{2}{3}\pi$ & $\Sigma_1^{2-}$ && \\
\cline{3-4}\cline{5-6}
&& \shortstrut $-\tfrac{2}{3}\pi$ & $\Sigma_1^{1+}$ \& $\Sigma_1^{2+}$ &
\multirow{2}{*}{$\mathbf{T}(\sqrt{2}\ee^{\frac{\ii\pi}{3}})$} &
\multirow{2}{*}{$D_\delta\setminus\circledomain$} \\
\cline{3-4}
&& \shortstrut $\pm\pi$ & $\Sigma_1^2$ && \\
\hline
\end{tabular}
\end{center}
\renewcommand{\arraystretch}{1}
\label{tab:Inner-y0-kappa0-s1}
\end{table}
\clearpage
\subsection{The gH case with $y_0\in\rectangle(\kappa)$ and $s=1$}
Here we use representative values of $y_0=0$ and $\kappa=0$.
\begin{figure}[h!]
\begin{subfigure}{\textwidth}
\begin{center}
\includegraphics[width=5in]{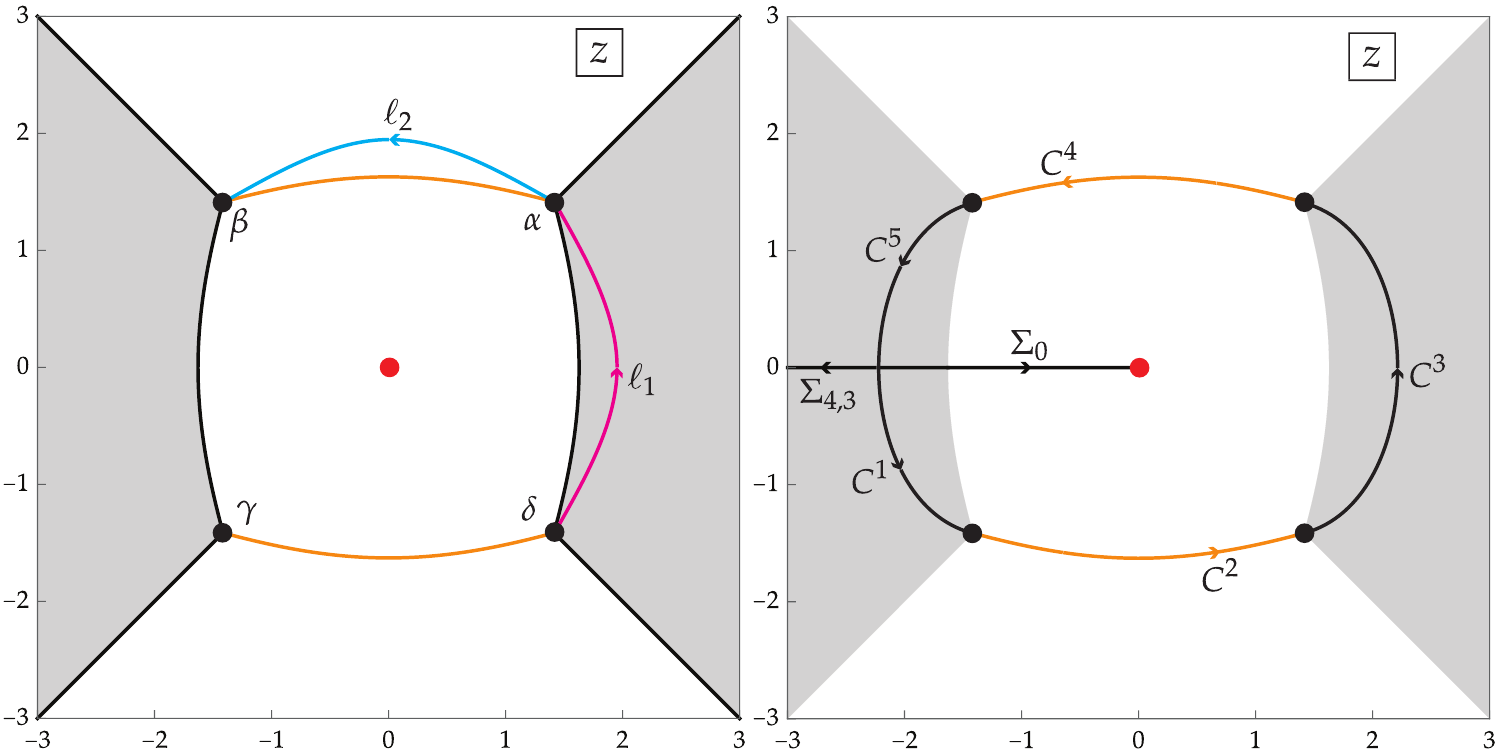}
\end{center}
\subcaption{Left:  Stokes graph including branch cuts (orange) for $h'(z)$ and contours $\ell_1$ and $\ell_2$ appearing in \eqref{eq:BoutrouxConstants}.  Right:  arcs of the jump contour $\Sigma$ for $\mathbf{M}(z)$; note that the arcs $C^2$ and $C^4$ lie on the Stokes graph.}
\label{fig:SignsContourLenses-y0-kappa0-s1-gH}
\end{subfigure}
\medskip
\begin{subfigure}{\textwidth}
\begin{center}
\includegraphics[width=4 in]{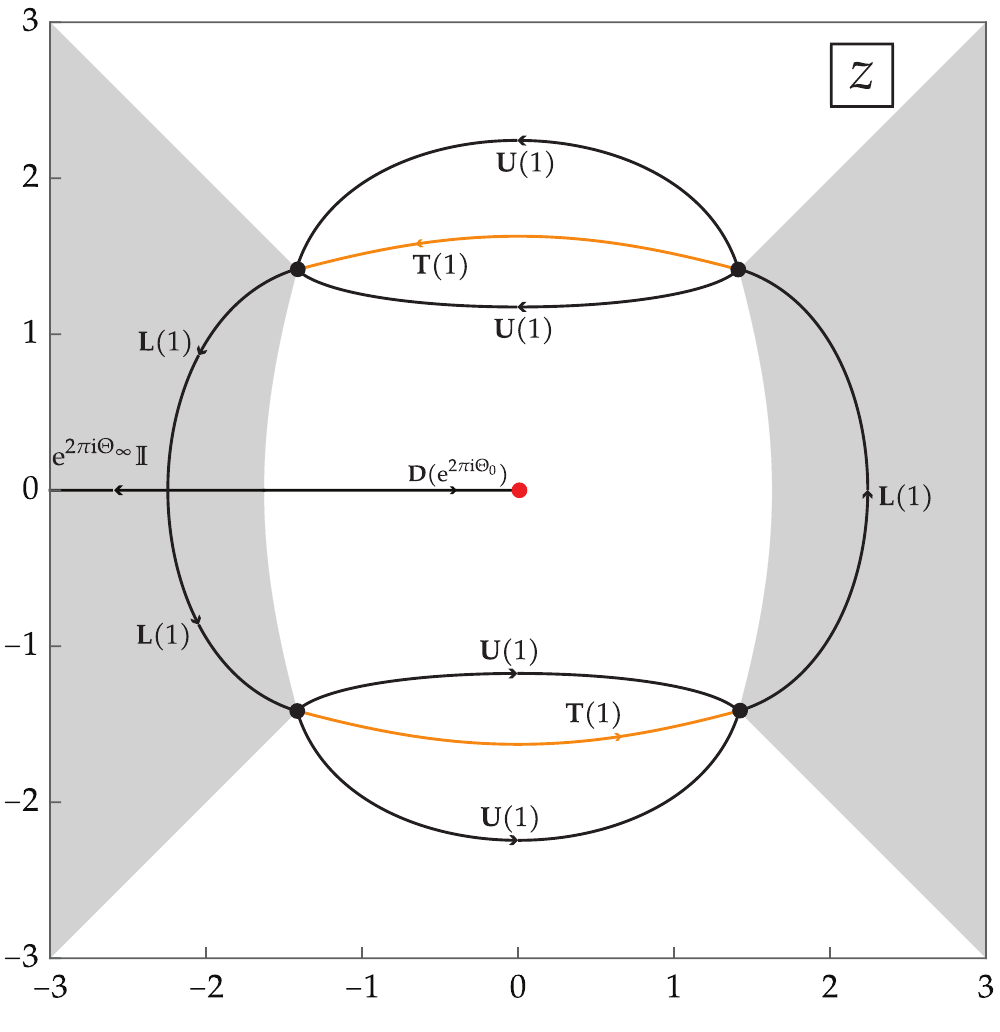}
\end{center}
\subcaption{The jump matrix $\mathbf{W}$ such that $\widetilde{\mathbf{O}}_+(z)=\widetilde{\mathbf{O}}_-(z)\mathbf{W}$.}
\label{fig:JustLenses-y0-kappa0-s1-gH}
\end{subfigure}
\caption{Diagrams for the (only) gH case with $y_0\in\rectangle(\kappa)$ and $s=1$.}
\end{figure}
\begin{figure}[h]\ContinuedFloat
\begin{subfigure}{\textwidth}
\begin{center}
\includegraphics[width=5 in]{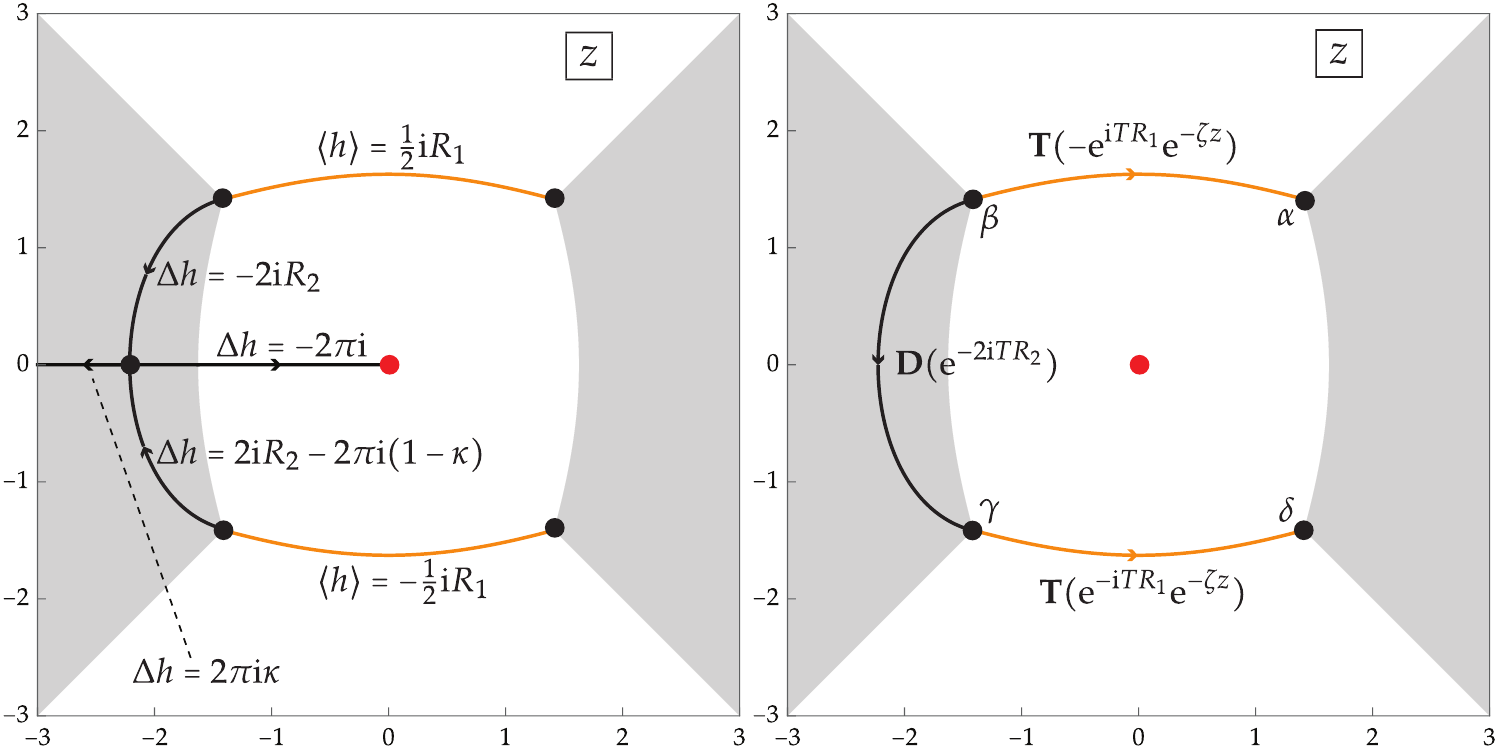}
\end{center}
\subcaption{Left:  the jump conditions satisfied by $h(z)$ with a suitable choice of integration constant.  Right:  the resulting jump conditions for the outer parametrix $\dot{\mathbf{O}}^{\mathrm{out}}(z)$.}
\label{fig:hJAO-y0-kappa0-s1-gH}
\end{subfigure}
\caption{Diagrams for the (only) gH case with $y_0\in\rectangle(\kappa)$ and $s=1$ (contd.)}
\end{figure}
\begin{table}
\caption{Inner parametrix data for the (only) gH case with $y_0\in\rectangle(\kappa)$ and $s=1$.}
\renewcommand{\arraystretch}{0.45}
\begin{center}
\begin{tabular}{@{}|l|l|c|c|c|c|@{}}
\hline
\multirow{2}{*}{$p$}&\multirow{2}{*}{Conformal map $W:D_p\to\mathbb{C}$}&\multicolumn{2}{c|}{Ray Preimages in $D_p$} & \multicolumn{2}{c|}{$\mathbf{C}(z)$ in $D_p$}\\
\cline{3-6}
& & $\arg(W)$ & Preimage & Value $\mathbf{C}$ & Subdomain of $D_p$ \\
\hline\hline
\multirow{4}{*}{$\alpha$} & \multirow{4}{*}{\makecell[l]{$(2 h(\alpha)-2 h(z))^{2/3}$,\\ 
continued from $C^3$}} &
\shortstrut $0$ & $C^3$ & 
\multirow{4}{*}{$\mathbf{D}(\ee^{-\frac{\ii\pi}{4}})$} & 
\multirow{4}{*}{$D_\alpha$} \\
\cline{3-4}
&& \shortstrut $\tfrac{2}{3}\pi$ & $C^{4-}$ && \\
\cline{3-4}
%\cline{5-6}
&& \shortstrut $-\tfrac{2}{3}\pi$ & $C^{4+}$ & 
& 
\\
\cline{3-4}
&& \shortstrut $\pm\pi$ & $C^4$ && \\
\hline
\hline
\multirow{4}{*}{$\beta$} & \multirow{4}{*}{\makecell[l]{$(2\langle h\rangle(\beta)-2\langle h\rangle(z))^{2/3}$,\\
continued from $C^5$}} & \shortstrut $0$ & $C^5$ &
\multirow{4}{*}{$\mathbf{D}(\ee^{\frac{\ii\pi}{4}})$} & 
\multirow{4}{*}{$D_\beta$} \\
\cline{3-4}
&& \shortstrut $\tfrac{2}{3}\pi$ & $C^{4+}$ && \\
\cline{3-4}
%\cline{5-6}
&& \shortstrut $-\tfrac{2}{3}\pi$ & $C^{4-}$ & 
 & \\
\cline{3-4}
&& \shortstrut $\pm\pi$ & $C^4$ && \\
\hline
\hline
\multirow{4}{*}{$\gamma$} & \multirow{4}{*}{\makecell[l]{$(2\langle h\rangle(\gamma)-2\langle h\rangle(z))^{2/3}$,\\
continued from $C^1$}} & \shortstrut $0$ & $C^1$ &
\multirow{4}{*}{$\mathbf{D}(\ee^{-\frac{\ii\pi}{4}})$} & 
\multirow{4}{*}{$D_\gamma$} \\
\cline{3-4}
&& \shortstrut $\tfrac{2}{3}\pi$ & $C^{2-}$ && \\
\cline{3-4}
%\cline{5-6}
&& \shortstrut $-\tfrac{2}{3}\pi$ & $C^{2+}$ & 
 & 
 \\
\cline{3-4}
&& \shortstrut $\pm\pi$ & $C^2$ && \\
\hline
\hline
\multirow{4}{*}{$\delta$} & \multirow{4}{*}{\makecell[l]{$(2h(\delta)-2h(z))^{2/3}$,\\
continued from $C^3$}} & \shortstrut $0$ & $C^3$ &
\multirow{4}{*}{$\mathbf{D}(\ee^{\frac{\ii\pi}{4}})$} & 
\multirow{4}{*}{$D_\delta\cap\circledomain$} \\
\cline{3-4}
&& \shortstrut $\tfrac{2}{3}\pi$ & $C^{2+}$ && \\
\cline{3-4}
%\cline{5-6}
&& \shortstrut $-\tfrac{2}{3}\pi$ & $C^{2-}$ &
 &
 \\
\cline{3-4}
&& \shortstrut $\pm\pi$ & $C^2$ && \\
\hline
\end{tabular}
\renewcommand{\arraystretch}{1}
\end{center}
\label{tab:Inner-y0-kappa0-s1-gH}
\end{table}
\clearpage

\subsection{The gO case with $y_0\in\rectangle(\kappa)$ and $s=-1$}
Here we use representative values of $y_0=0$ and $\kappa=0$.
\begin{figure}[h!]
\begin{subfigure}{\textwidth}
\begin{center}
\includegraphics[width=5 in]{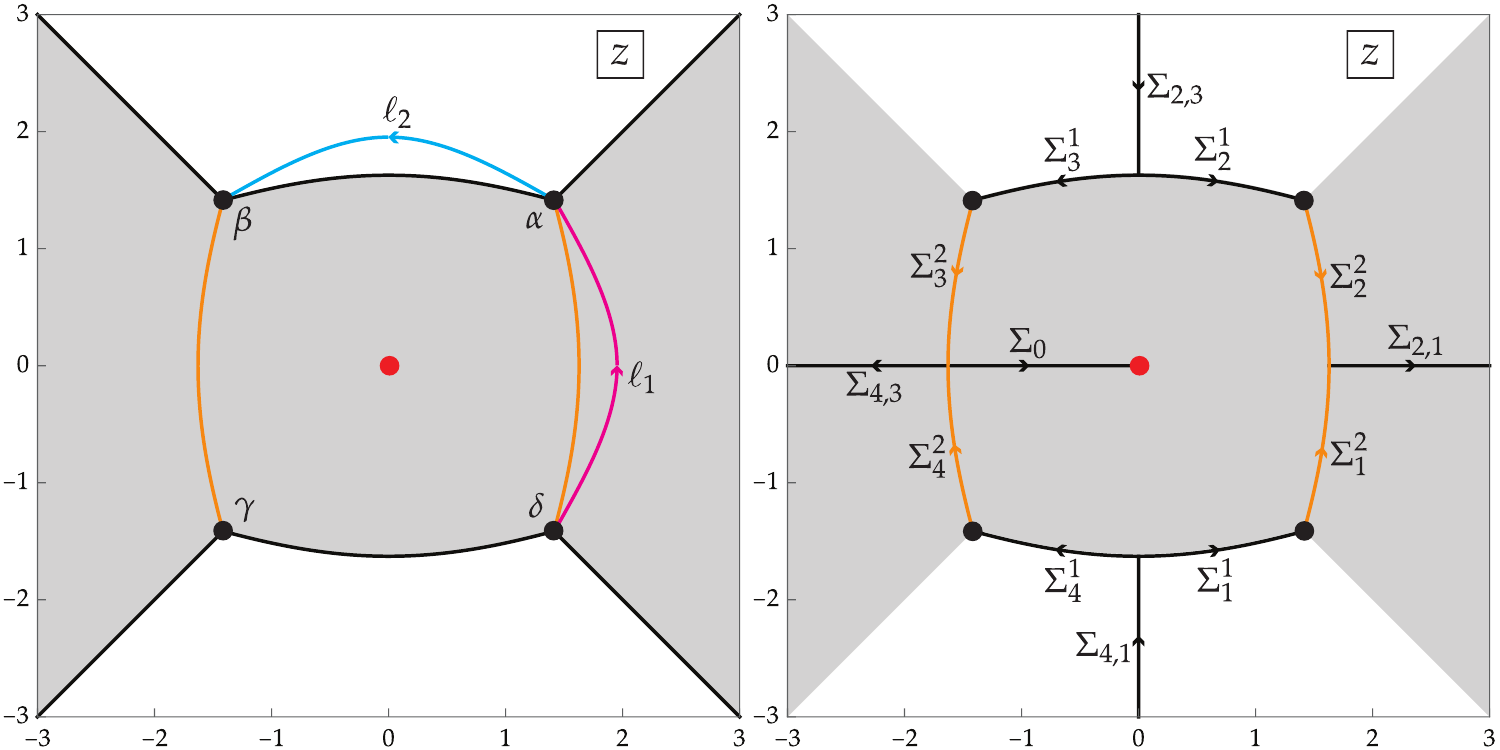}
\end{center}
\subcaption{Left:  Stokes graph including branch cuts (orange) for $h'(z)$ and contours $\ell_1$ and $\ell_2$ appearing in \eqref{eq:BoutrouxConstants}.  Right:
arcs of the jump contour $\Sigma$ for $\mathbf{M}(z)$; note that the arcs $\Sigma_j^k$, $j=1,\dots,4$ and $k=1,2$, lie on the Stokes graph. }
\label{fig:SignsContourLenses-y0-kappa0-sm1}
\end{subfigure}
\medskip
\begin{subfigure}{\textwidth}
\begin{center}
\includegraphics[width=4 in]{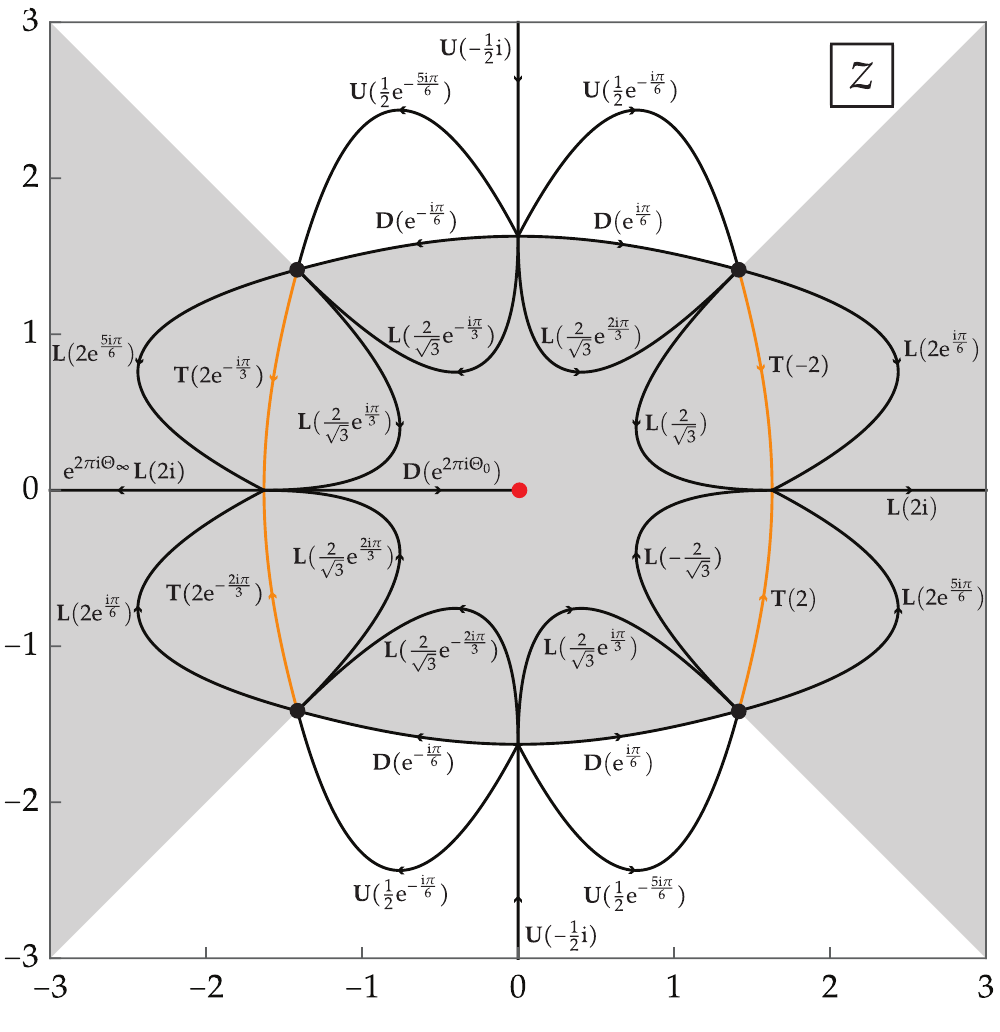}
\end{center}
\subcaption{The jump matrix $\mathbf{W}$ such that $\widetilde{\mathbf{O}}_+(z)=\widetilde{\mathbf{O}}_-(z)\mathbf{W}$.}
\label{fig:JustLenses-y0-kappa0-sm1}
\end{subfigure}
\caption{Diagrams for the gO case with $y_0\in\rectangle(\kappa)$ and $s=-1$.}
\end{figure}
\begin{figure}[h]\ContinuedFloat
\begin{subfigure}{\textwidth}
\begin{center}
\includegraphics[width=5 in]{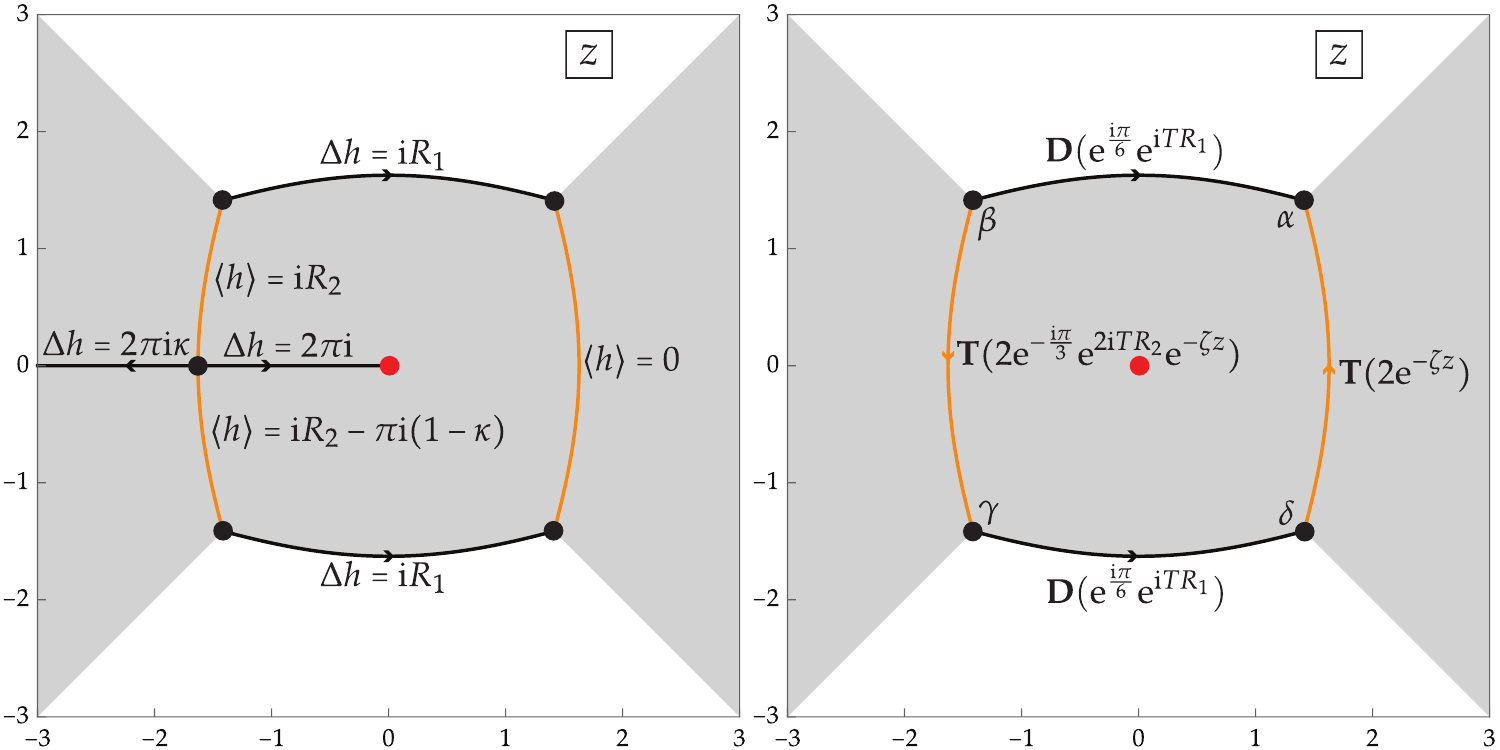}
\end{center}
\subcaption{Left:  the jump conditions satisfied by $h(z)$ with a suitable choice of integration constant.  Right:  the resulting jump conditions for the outer parametrix $\dot{\mathbf{O}}^{\mathrm{out}}(z)$.}
\label{fig:hJAO-y0-kappa0-sm1}
\end{subfigure}
\caption{Diagrams for the gO case with $y_0\in\rectangle(\kappa)$ and $s=-1$ (contd.)}
\end{figure}
\begin{table}
\caption{Inner parametrix data for the gO case with $y_0\in\rectangle(\kappa)$ and $s=-1$.}
\renewcommand{\arraystretch}{0.45}
\begin{center}
\begin{tabular}{@{}|l|l|c|c|c|c|@{}}
\hline
\multirow{2}{*}{$p$}&\multirow{2}{*}{Conformal map $W:D_p\to\mathbb{C}$}&\multicolumn{2}{c|}{Ray Preimages in $D_p$} & \multicolumn{2}{c|}{$\mathbf{C}(z)$ in $D_p$}\\
\cline{3-6}
& & $\arg(W)$ & Preimage & Value $\mathbf{C}$ & Subdomain of $D_p$ \\
\hline\hline
\multirow{4}{*}{$\alpha$} & \multirow{4}{*}{\makecell[l]{$(2 h(z)-2 h(\alpha))^{2/3}$,\\ 
continued from $\Sigma_2^{1+}$; $h(\alpha)$\\ defined by limit along $\Sigma_2^{1+}$}} &
\shortstrut $0$ & $\Sigma_2^{1+}$ & 
\multirow{2}{*}{$\mathbf{T}(\sqrt{2}\ee^{\frac{\ii\pi}{6}})$} & 
\multirow{2}{*}{$D_\alpha\cap\circledomain$} \\
\cline{3-4}
&& \shortstrut $\tfrac{2}{3}\pi$ & $\Sigma_2^{1-}$ \& $\Sigma_2^{2-}$ && \\
\cline{3-4}\cline{5-6}
&& \shortstrut $-\tfrac{2}{3}\pi$ & $\Sigma_2^{2+}$ & 
\multirow{2}{*}{$\mathbf{T}(\sqrt{2}\ee^{\frac{\ii\pi}{3}})$} & 
\multirow{2}{*}{$D_\alpha\setminus\circledomain$} \\
\cline{3-4}
&& \shortstrut $\pm\pi$ & $\Sigma_2^2$ && \\
\hline
\hline
\multirow{4}{*}{$\beta$} & \multirow{4}{*}{\makecell[l]{$(2h(z)-2h(\beta))^{2/3}$,\\
continued from $\Sigma_3^{1-}$; $h(\beta)$\\
defined by limit along $\Sigma_3^{1-}$}} & \shortstrut $0$ & $\Sigma_3^{1-}$ &
\multirow{2}{*}{$\mathbf{T}(\sqrt{2}\ii)$} & 
\multirow{2}{*}{$D_\beta\cap\circledomain$} \\
\cline{3-4}
&& \shortstrut $\tfrac{2}{3}\pi$ & $\Sigma_3^{2-}$ && \\
\cline{3-4}\cline{5-6}
&& \shortstrut $-\tfrac{2}{3}\pi$ & $\Sigma_3^{1+}$ \& $\Sigma_3^{2+}$ & 
\multirow{2}{*}{$\mathbf{T}(\sqrt{2}\ee^{\frac{2\ii\pi}{3}})$} & 
\multirow{2}{*}{$D_\beta\setminus\circledomain$} \\
\cline{3-4}
&& \shortstrut $\pm\pi$ & $\Sigma_3^2$ && \\
\hline
\hline
\multirow{4}{*}{$\gamma$} & \multirow{4}{*}{\makecell[l]{$(2h(z)-2h(\gamma))^{2/3}$,\\
continued from $\Sigma_4^{1+}$; $h(\gamma)$ \\
defined by limit along $\Sigma_4^{1+}$}} & \shortstrut $0$ & $\Sigma_4^{1+}$ &
\multirow{2}{*}{$\mathbf{T}(\sqrt{2}\ii)$} & 
\multirow{2}{*}{$D_\gamma\cap\circledomain$} \\
\cline{3-4}
&& \shortstrut $\tfrac{2}{3}\pi$ & $\Sigma_4^{1-}$ \& $\Sigma_4^{2-}$ && \\
\cline{3-4}\cline{5-6}
&& \shortstrut $-\tfrac{2}{3}\pi$ & $\Sigma_4^{2+}$ & 
\multirow{2}{*}{$\mathbf{T}(\sqrt{2}\ee^{\frac{\ii\pi}{3}})$} & 
\multirow{2}{*}{$D_\gamma\setminus\circledomain$} \\
\cline{3-4}
&& \shortstrut $\pm\pi$ & $\Sigma_4^2$ && \\
\hline
\hline
\multirow{4}{*}{$\delta$} & \multirow{4}{*}{\makecell[l]{$(2h(z)-2h(\delta))^{2/3}$,\\
continued from $\Sigma_1^{1-}$; $h(\delta)$ \\
defined by limit along $\Sigma_1^{1-}$}} & \shortstrut $0$ & $\Sigma_1^{1-}$ &
\multirow{2}{*}{$\mathbf{T}(\sqrt{2}\ee^{\frac{5\ii\pi}{6}})$} & 
\multirow{2}{*}{$D_\delta\cap\circledomain$} \\
\cline{3-4}
&& \shortstrut $\tfrac{2}{3}\pi$ & $\Sigma_1^{2-}$ && \\
\cline{3-4}\cline{5-6}
&& \shortstrut $-\tfrac{2}{3}\pi$ & $\Sigma_1^{1+}$ \& $\Sigma_1^{2+}$ &
\multirow{2}{*}{$\mathbf{T}(\sqrt{2}\ee^{\frac{2\ii\pi}{3}})$} &
\multirow{2}{*}{$D_\delta\setminus\circledomain$} \\
\cline{3-4}
&& \shortstrut $\pm\pi$ & $\Sigma_1^2$ && \\
\hline
\end{tabular}
\end{center}
\renewcommand{\arraystretch}{1}
\label{tab:Inner-y0-kappa0-sm1}
\end{table}
\clearpage

%\begin{figure}[h!]
%\begin{center}
%\includegraphics{SignsContourLenses-y0-kappa1p05-s1.pdf}
%\end{center}
%\caption{For $y_0=0$, $\kappa=1.05$ (representative of $y_0\in\rectangle$, $\kappa>1$), and $s=1$.  Top left:  Stokes graph including branch cuts (orange) for $h'(z)$ and contours $\ell_1$ and $\ell_2$.  Top right:  arcs of the jump contour $\Sigma$ for $\mathbf{M}(z)$; note that the arcs $\Sigma_j$, $j=1,...,4$, $\Sigma_{2,3}^2$, and $\Sigma_{4,1}^2$ lie on the Stokes graph.  Below: the jump matrix $\mathbf{W}$ such that $\tilde{\mathbf{O}}_+(z)=\tilde{\mathbf{O}}_-(z)\mathbf{W}$.  In all plots: gray shading means $\mathrm{Re}(h(z))<0$ and white background means $\mathrm{Re}(h(z))>0$.}
%\label{fig:SignsContourLenses-y0-kappa1p05-s1}
%\end{figure}
%\clearpage
%\begin{figure}[h!]
%\begin{center}
%\includegraphics{SignsContourLenses-y0-kappa1p05-sm1.pdf}
%\end{center}
%\caption{As in Figure~\ref{fig:SignsContourLenses-y0-kappa1p05-s1} but for $s=-1$ instead.}
%\label{fig:SignsContourLenses-y0-kappa1p05-sm1}
%\end{figure}
%\clearpage

\subsection{The gO case with $y_0\in\TR(\kappa)$ and $s=1$}
Here we use representative values of $y_0=1.3$ and $\kappa=0$.
\begin{figure}[h!]
\begin{subfigure}{\textwidth}
\begin{center}
\includegraphics[width=5 in]{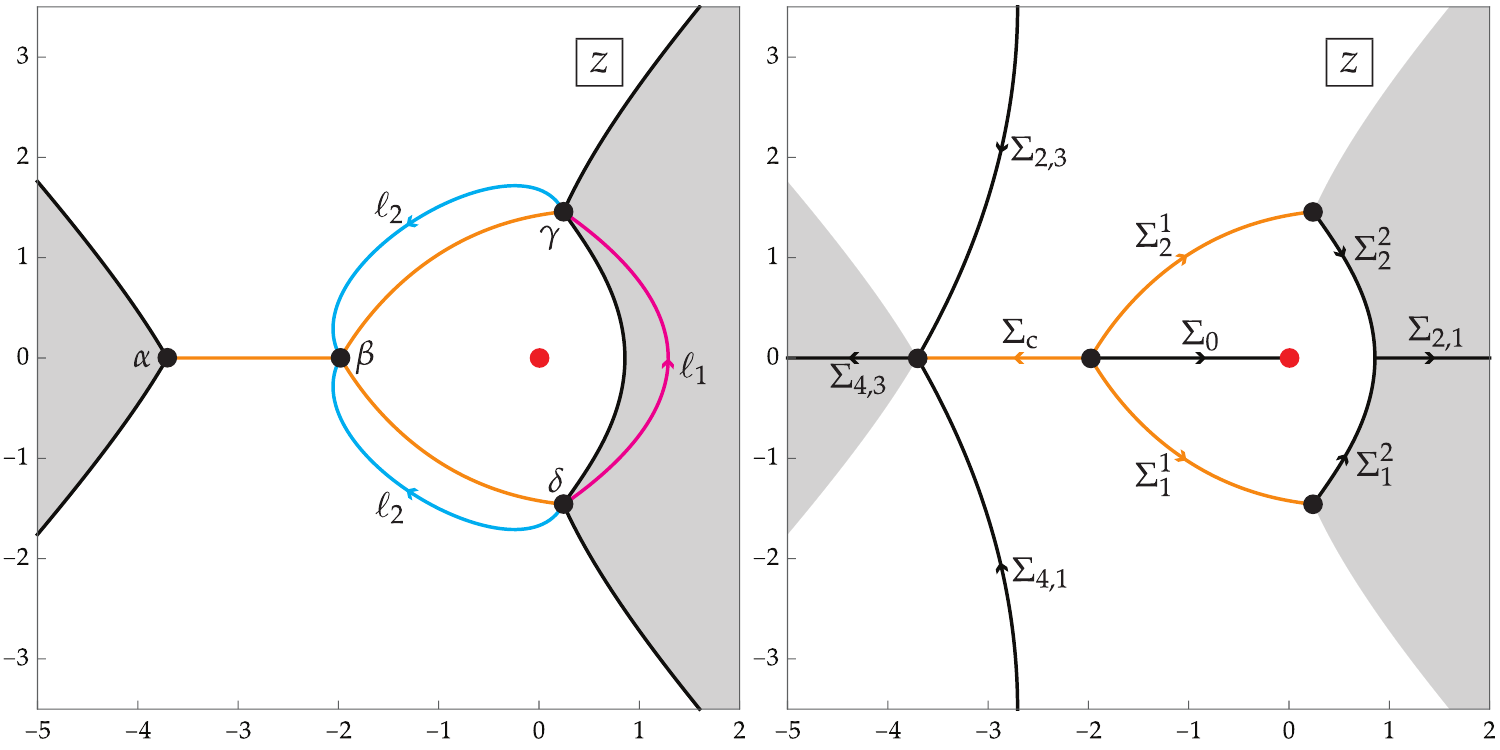}
\end{center}
\subcaption{Left:  Stokes graph including branch cuts (orange) for $h'(z)$ and contours $\ell_1$ and $\ell_2$ appearing in \eqref{eq:BoutrouxConstants}.  Right:  arcs of the jump contour $\Sigma$ for $\mathbf{M}(z)$; note that the arcs $\Sigma_j^k$, $j=1,2$ and $k=1,2$, and $\Sigma_\mathrm{c}$ lie on the Stokes graph.}
\label{fig:SignsContourLenses-y1p3-kappa0-s1}
\end{subfigure}
\medskip
\begin{subfigure}{\textwidth}
\begin{center}
\includegraphics[width=4 in]{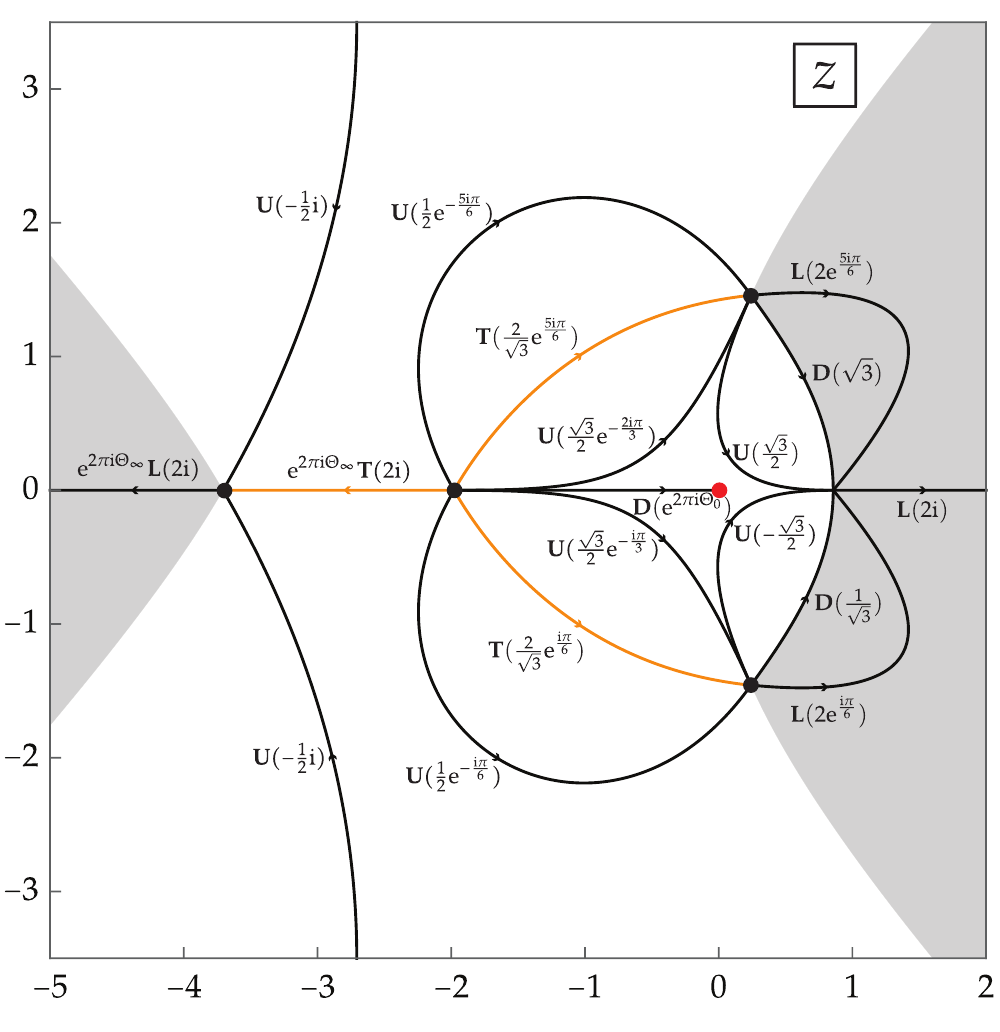}
\end{center}
\subcaption{The jump matrix $\mathbf{W}$ such that $\widetilde{\mathbf{O}}_+(z)=\widetilde{\mathbf{O}}_-(z)\mathbf{W}$.}
\label{fig:JustLenses-y1p3-kappa0-s1}
\end{subfigure}
\caption{Diagrams for the gO case with $y_0\in\TR(\kappa)$ and $s=1$.}
\end{figure}
\begin{figure}[h]\ContinuedFloat
\begin{subfigure}{\textwidth}
\begin{center}
\includegraphics[width=5 in]{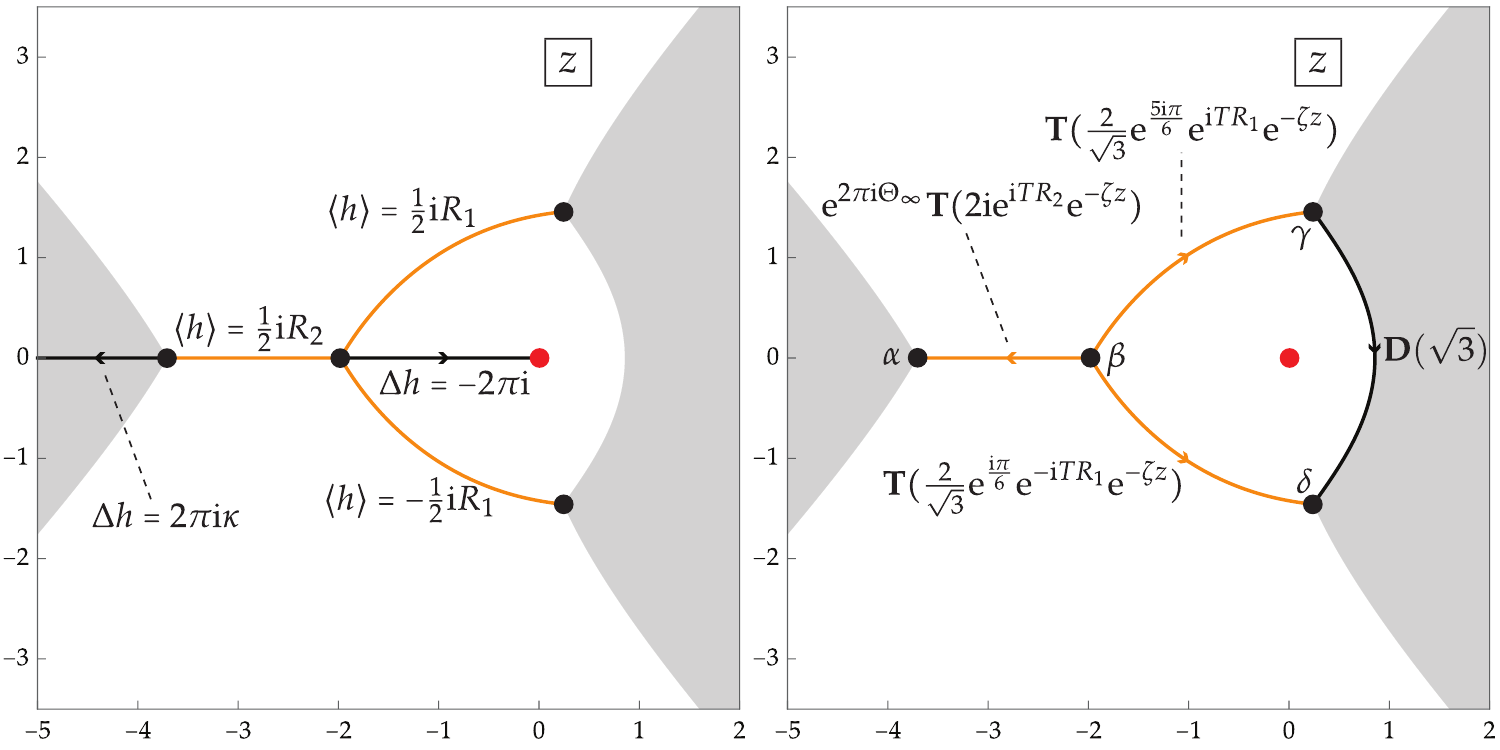}
\end{center}
\subcaption{Left:  the jump conditions satisfied by $h(z)$ with a suitable choice of integration constant.  Right:  the resulting jump conditions for the outer parametrix $\dot{\mathbf{O}}^{\mathrm{out}}(z)$.}
\label{fig:hJAO-y1p3-kappa0-s1}
\end{subfigure}
\caption{Diagrams for the gO case with $y_0\in\TR(\kappa)$ and $s=1$ (contd.)}
\end{figure}
\begin{table}
\caption{Inner parametrix data for the gO case with $y_0\in\TR(\kappa)$ and $s=1$.}
\renewcommand{\arraystretch}{0.45}
\begin{center}
\begin{tabular}{@{}|l|l|c|c|c|c|@{}}
\hline
\multirow{2}{*}{$p$}&\multirow{2}{*}{Conformal map $W:D_p\to\mathbb{C}$}&\multicolumn{2}{c|}{Ray Preimages in $D_p$} & \multicolumn{2}{c|}{$\mathbf{C}(z)$ in $D_p$}\\
\cline{3-6}
& & $\arg(W)$ & Preimage & Value $\mathbf{C}$ & Subdomain of $D_p$ \\
\hline\hline
\multirow{4}{*}{$\alpha$} & \multirow{4}{*}{\makecell[l]{$(2\langle h\rangle(\alpha)-2\langle h\rangle(z))^{2/3}$,\\ 
continued from $\Sigma_{4,3}$
%; $\langle h\rangle(\alpha)$\\ defined by limit along $\Sigma_{4,3}$
}} &
\shortstrut $0$ & $\Sigma_{4,3}$ & 
\multirow{2}{*}{$\mathbf{D}(\tfrac{1}{\sqrt{2}})$} & 
\multirow{2}{*}{$D_\alpha$, left of $\Sigma_{4,3}$ \& $\Sigma_\mathrm{c}$} \\
\cline{3-4}
&& \shortstrut $\tfrac{2}{3}\pi$ & $\Sigma_{4,1}$ && \\
\cline{3-4}\cline{5-6}
&& \shortstrut $-\tfrac{2}{3}\pi$ & $\Sigma_{2,3}$ & 
\multirow{2}{*}{$\ee^{2\pi\ii\Theta_\infty}\mathbf{D}(\tfrac{1}{\sqrt{2}})$} & 
\multirow{2}{*}{$D_\alpha$, right of $\Sigma_{4,3}$ \& $\Sigma_\mathrm{c}$} \\
\cline{3-4}
&& \shortstrut $\pm\pi$ & $\Sigma_\mathrm{c}$ && \\
\hline
\hline
\multirow{4}{*}{$\beta$} & \multirow{4}{*}{\makecell[l]{$(2\langle h\rangle(z)-2\langle h\rangle(\beta))^{2/3}$,\\
continued from $\Sigma_0$; $\langle h\rangle(\beta)$\\
defined by limit along $\Sigma_0$}} & \shortstrut $0$ & $\Sigma_1^{1+}$, $\Sigma_0$, \& $\Sigma_2^{1-}$ &
$\mathbf{T}(\sqrt{\tfrac{2}{3}}\ee^{-\frac{5\ii\pi}{6}})$ & $D_\beta\cap\circledomain$, left of $\Sigma_0$ \\
\cline{3-4}\cline{5-6}
&& \shortstrut $\tfrac{2}{3}\pi$ & $\Sigma_2^{1+}$ & $\mathbf{T}(\sqrt{\tfrac{2}{3}}\ee^{-2\pi\ii\Theta_0-\frac{5\ii\pi}{6}})$ & $D_\beta\cap\circledomain$, right of $\Sigma_0$ \\
\cline{3-4}\cline{5-6}
&& \shortstrut $-\tfrac{2}{3}\pi$ & $\Sigma_1^{1-}$ & $\mathbf{D}(\tfrac{1}{\sqrt{2}}\ee^{-2\pi\ii\Theta_0})$ & $D_\beta\setminus\circledomain$, left of $\Sigma_\mathrm{c}$ \\
\cline{3-4}\cline{5-6}
&& \shortstrut $\pm\pi$ & $\Sigma_\mathrm{c}$ & $\mathbf{D}(\tfrac{1}{\sqrt{2}}\ee^{\frac{\ii\pi}{3}})$ &
$D_\beta\setminus\circledomain$, right of $\Sigma_\mathrm{c}$ \\
\hline
\hline
\multirow{4}{*}{$\gamma$} & \multirow{4}{*}{\makecell[l]{$(2h(z)-2h(\gamma))^{2/3}$,\\
continued from $\Sigma_2^{1+}$
%; $h(\gamma)$ \\defined by limit along $\Sigma_2^{1+}$
}} & \shortstrut $0$ & $\Sigma_2^{1+}$ &
\multirow{2}{*}{$\mathbf{D}(\sqrt{\tfrac{3}{2}}\ee^{\frac{5\ii\pi}{6}})$} & 
\multirow{2}{*}{$D_\gamma\cap\circledomain$} \\
\cline{3-4}
&& \shortstrut $\tfrac{2}{3}\pi$ & $\Sigma_2^{1-}$ \& $\Sigma_2^{2-}$ && \\
\cline{3-4}\cline{5-6}
&& \shortstrut $-\tfrac{2}{3}\pi$ & $\Sigma_2^{2+}$ & 
\multirow{2}{*}{$\mathbf{T}(\sqrt{2}\ee^{\frac{2\ii\pi}{3}})$} & 
\multirow{2}{*}{$D_\gamma\setminus\circledomain$} \\
\cline{3-4}
&& \shortstrut $\pm\pi$ & $\Sigma_2^2$ && \\
\hline
\hline
\multirow{4}{*}{$\delta$} & \multirow{4}{*}{\makecell[l]{$(2h(z)-2h(\delta))^{2/3}$,\\
continued from $\Sigma_1^{1-}$
%; $h(\delta)$ \\ defined by limit along $\Sigma_1^{1-}$
}} & \shortstrut $0$ & $\Sigma_1^{1-}$ &
\multirow{2}{*}{$\mathbf{D}(\sqrt{\tfrac{3}{2}}\ee^{\frac{\ii\pi}{6}})$} & 
\multirow{2}{*}{$D_\delta\cap\circledomain$} \\
\cline{3-4}
&& \shortstrut $\tfrac{2}{3}\pi$ & $\Sigma_1^{2-}$ && \\
\cline{3-4}\cline{5-6}
&& \shortstrut $-\tfrac{2}{3}\pi$ & $\Sigma_1^{1+}$ \& $\Sigma_1^{2+}$ &
\multirow{2}{*}{$\mathbf{T}(\sqrt{2}\ee^{\frac{\ii\pi}{3}})$} &
\multirow{2}{*}{$D_\delta\setminus\circledomain$} \\
\cline{3-4}
&& \shortstrut $\pm\pi$ & $\Sigma_1^2$ && \\
\hline
\end{tabular}
\end{center}
\renewcommand{\arraystretch}{1}
\label{tab:Inner-y1p3-kappa0-s1}
\end{table}
\clearpage

\subsection{The gO case with $y_0\in\TR(\kappa)$ and $s=-1$}
Here we use representative values of $y_0=1.3$ and $\kappa=0$.
\begin{figure}[h!]
\begin{subfigure}{\textwidth}
\begin{center}
\includegraphics[width=5 in]{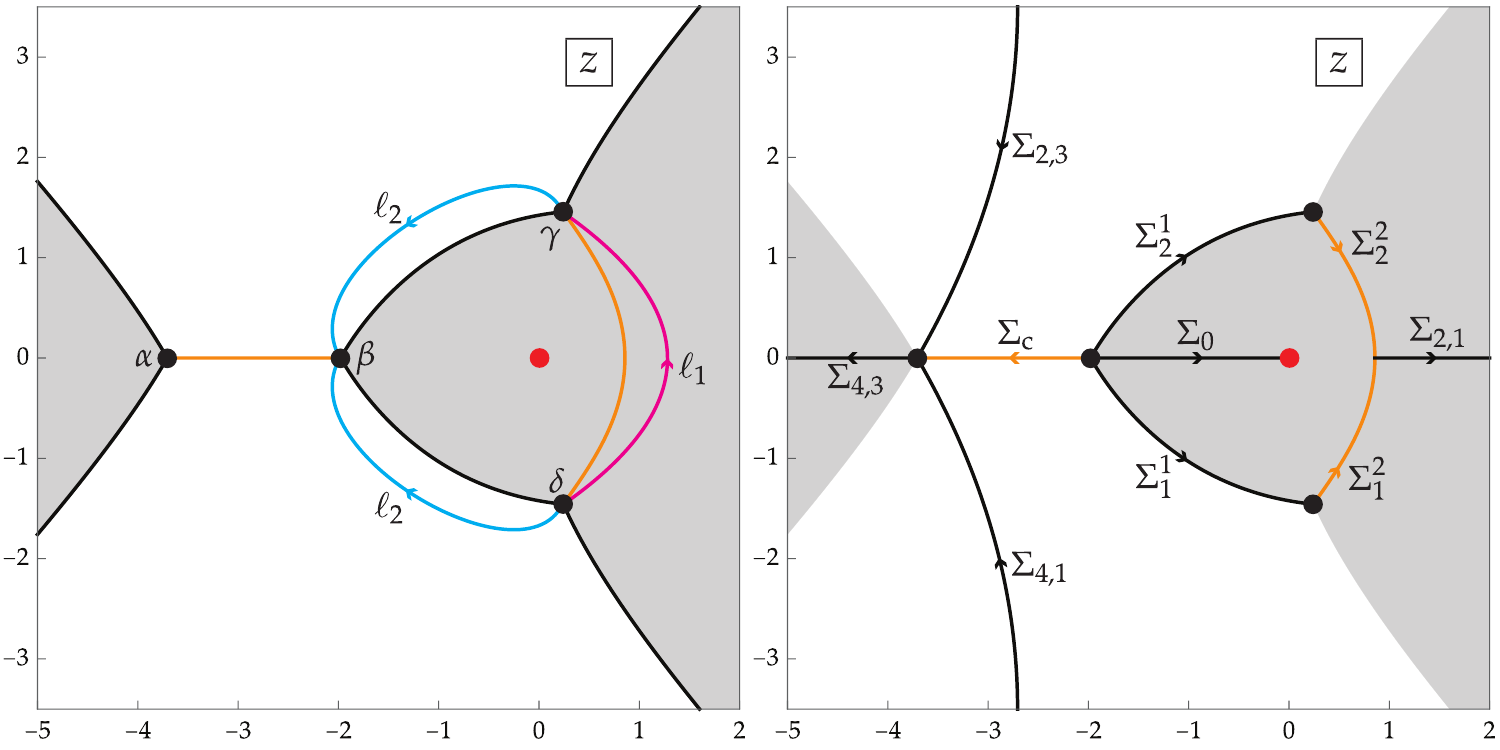}
\end{center}
\subcaption{Left:  Stokes graph including branch cuts (orange) for $h'(z)$ and contours $\ell_1$ and $\ell_2$ appearing in \eqref{eq:BoutrouxConstants}.  Right:  arcs of the jump contour $\Sigma$ for $\mathbf{M}(z)$; note that the arcs $\Sigma_j^k$, $j=1,2$ and $k=1,2$, and $\Sigma_\mathrm{c}$ lie on the Stokes graph.}
\label{fig:SignsContourLenses-y1p3-kappa0-sm1}
\end{subfigure}
\medskip
\begin{subfigure}{\textwidth}
\begin{center}
\includegraphics[width=4 in]{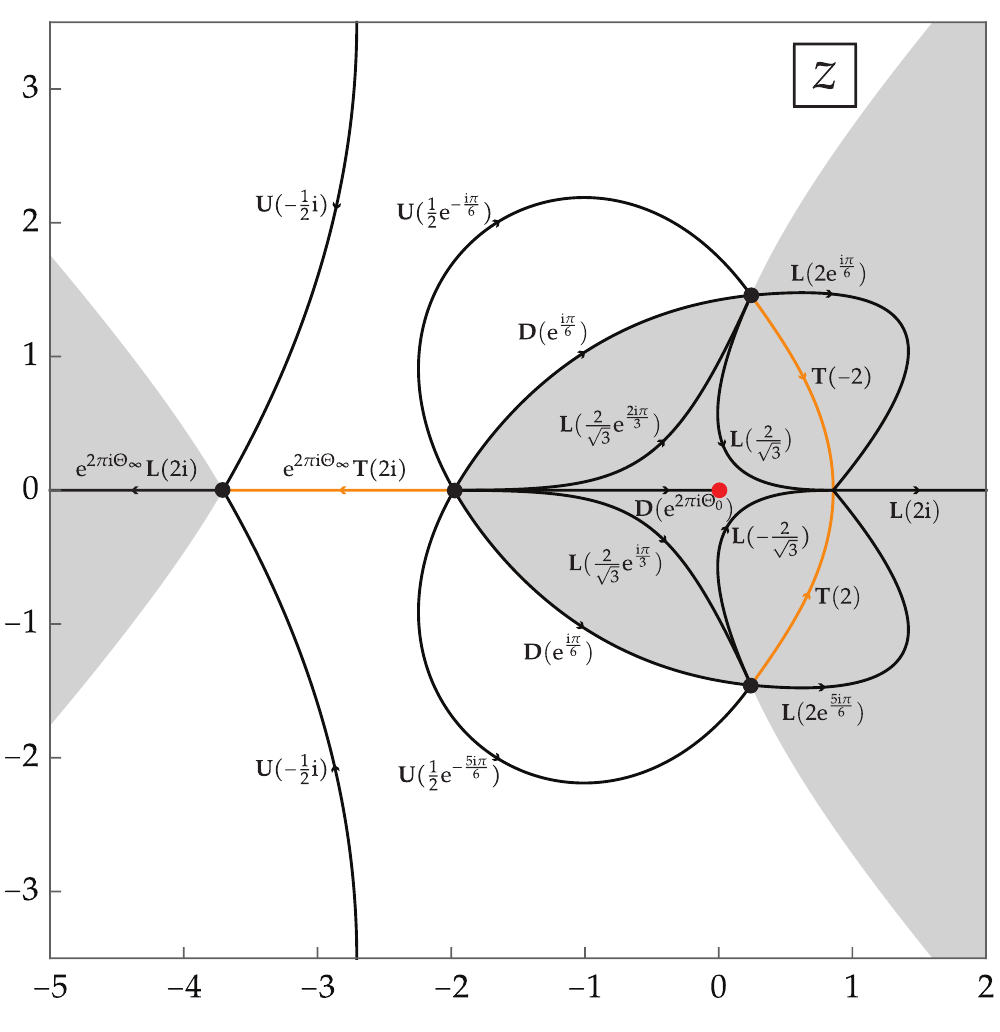}
\end{center}
\subcaption{The jump matrix $\mathbf{W}$ such that $\widetilde{\mathbf{O}}_+(z)=\widetilde{\mathbf{O}}_-(z)\mathbf{W}$.}
\label{fig:JustLenses-y1p3-kappa0-sm1}
\end{subfigure}
\caption{Diagrams for the gO case with $y_0\in\TR(\kappa)$ and $s=-1$.}
\end{figure}
\begin{figure}[h]\ContinuedFloat
\begin{subfigure}{\textwidth}
\begin{center}
\includegraphics[width=5 in]{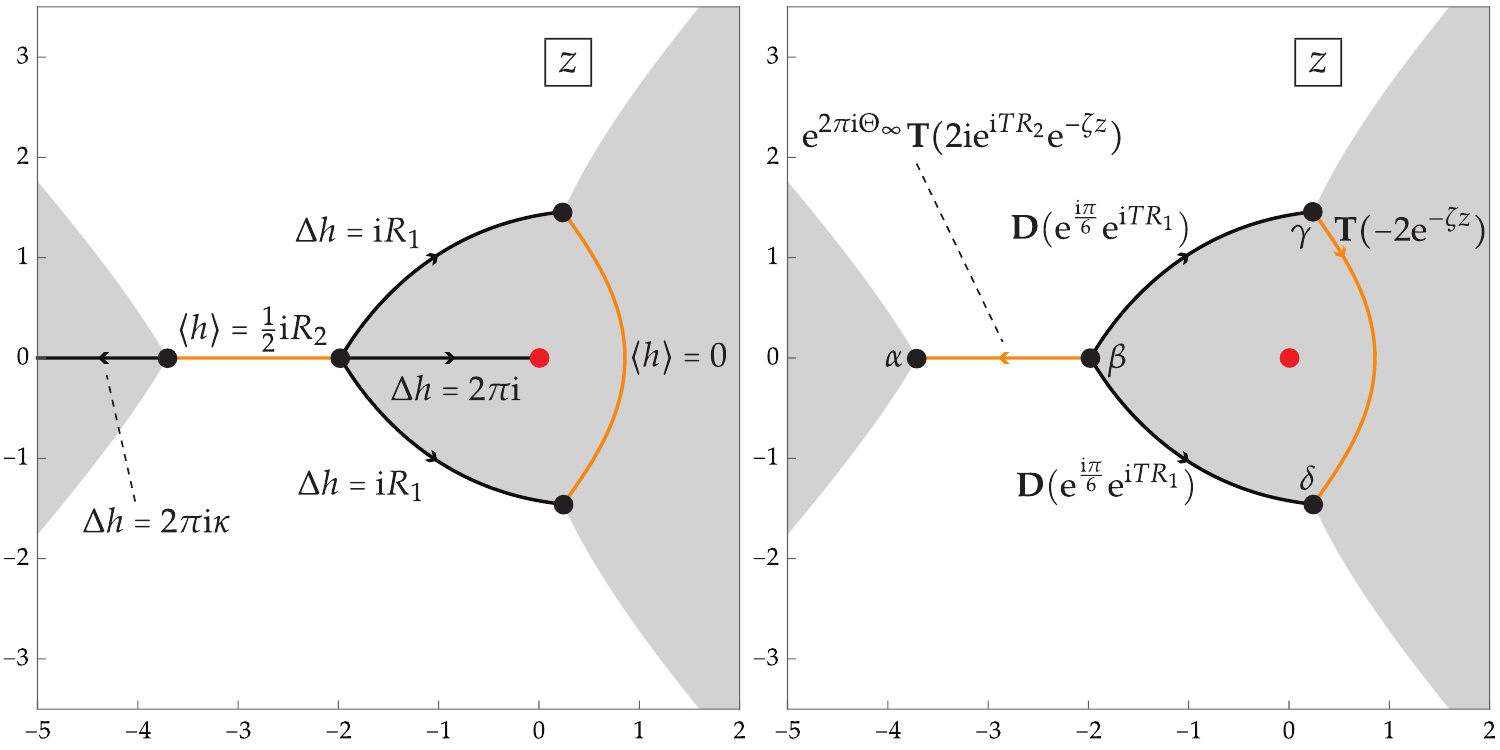}
\end{center}
\subcaption{Left:  the jump conditions satisfied by $h(z)$ with a suitable choice of integration constant.  Right:  the resulting jump conditions for the outer parametrix $\dot{\mathbf{O}}^{\mathrm{out}}(z)$.}
\label{fig:hJAO-y1p3-kappa0-sm1}
\end{subfigure}
\caption{Diagrams for the gO case with $y_0\in\TR(\kappa)$ and $s=-1$ (contd.)}
\end{figure}
\begin{table}
\caption{Inner parametrix data for the gO case with $y_0\in\TR(\kappa)$ and $s=-1$.}
\renewcommand{\arraystretch}{0.45}
\begin{center}
\begin{tabular}{@{}|l|l|c|c|c|c|@{}}
\hline
\multirow{2}{*}{$p$}&\multirow{2}{*}{Conformal map $W:D_p\to\mathbb{C}$}&\multicolumn{2}{c|}{Ray Preimages in $D_p$} & \multicolumn{2}{c|}{$\mathbf{C}(z)$ in $D_p$}\\
\cline{3-6}
& & $\arg(W)$ & Preimage & Value $\mathbf{C}$ & Subdomain of $D_p$ \\
\hline\hline
\multirow{4}{*}{$\alpha$} & \multirow{4}{*}{\makecell[l]{$(2\langle h\rangle(\alpha)-2\langle h\rangle(z))^{2/3}$,\\ 
continued from $\Sigma_{4,3}$
%; $\langle h\rangle(\alpha)$\\ defined by limit along $\Sigma_{4,3}$
}} &
\shortstrut $0$ & $\Sigma_{4,3}$ & 
\multirow{2}{*}{$\mathbf{D}(\tfrac{1}{\sqrt{2}})$} & 
\multirow{2}{*}{$D_\alpha$, left of $\Sigma_{4,3}$ \& $\Sigma_\mathrm{c}$} \\
\cline{3-4}
&& \shortstrut $\tfrac{2}{3}\pi$ & $\Sigma_{4,1}$ && \\
\cline{3-4}\cline{5-6}
&& \shortstrut $-\tfrac{2}{3}\pi$ & $\Sigma_{2,3}$ & 
\multirow{2}{*}{$\ee^{2\pi\ii\Theta_\infty}\mathbf{D}(\tfrac{1}{\sqrt{2}})$} & 
\multirow{2}{*}{$D_\alpha$, right of $\Sigma_{4,3}$ \& $\Sigma_\mathrm{c}$} \\
\cline{3-4}
&& \shortstrut $\pm\pi$ & $\Sigma_\mathrm{c}$ && \\
\hline
\hline
\multirow{4}{*}{$\beta$} & \multirow{4}{*}{\makecell[l]{$(2\langle h\rangle(\beta)-2\langle h\rangle(z))^{2/3}$,\\
continued from $\Sigma_0$; $\langle h\rangle(\beta)$\\
defined by limit along $\Sigma_0$}} & \shortstrut $0$ & $\Sigma_1^{1+}$, $\Sigma_0$, \& $\Sigma_2^{1-}$ &
$\mathbf{D}(\tfrac{1}{\sqrt{2}}\ee^{-\frac{\ii\pi}{6}})$ & $D_\beta\cap\circledomain$, left of $\Sigma_0$ \\
\cline{3-4}\cline{5-6}
&& \shortstrut $\tfrac{2}{3}\pi$ & $\Sigma_2^{1+}$ & $\mathbf{D}(\tfrac{1}{\sqrt{2}}\ee^{2\pi\ii\Theta_0-\frac{\ii\pi}{6}})$ & $D_\beta\cap\circledomain$, right of $\Sigma_0$ \\
\cline{3-4}\cline{5-6}
&& \shortstrut $-\tfrac{2}{3}\pi$ & $\Sigma_1^{1-}$ & $\mathbf{D}(\tfrac{1}{\sqrt{2}}\ee^{2\pi\ii\Theta_0})$ & $D_\beta\setminus\circledomain$, left of $\Sigma_\mathrm{c}$ \\
\cline{3-4}\cline{5-6}
&& \shortstrut $\pm\pi$ & $\Sigma_\mathrm{c}$ & $\mathbf{D}(\tfrac{1}{\sqrt{2}}\ee^{-\frac{\ii\pi}{3}})$ &
$D_\beta\setminus\circledomain$, right of $\Sigma_\mathrm{c}$ \\
\hline
\hline
\multirow{4}{*}{$\gamma$} & \multirow{4}{*}{\makecell[l]{$(2h(z)-2h(\gamma))^{2/3}$,\\
continued from $\Sigma_2^{1+}$; $h(\gamma)$ \\
defined by limit along $\Sigma_2^{1+}$}} & \shortstrut $0$ & $\Sigma_2^{1+}$ &
\multirow{2}{*}{$\mathbf{T}(\sqrt{2}\ee^{\frac{\ii\pi}{6}})$} & 
\multirow{2}{*}{$D_\gamma\cap\circledomain$} \\
\cline{3-4}
&& \shortstrut $\tfrac{2}{3}\pi$ & $\Sigma_2^{1-}$ \& $\Sigma_2^{2-}$ && \\
\cline{3-4}\cline{5-6}
&& \shortstrut $-\tfrac{2}{3}\pi$ & $\Sigma_2^{2+}$ & 
\multirow{2}{*}{$\mathbf{T}(\sqrt{2}\ee^{\frac{\ii\pi}{3}})$} & 
\multirow{2}{*}{$D_\gamma\setminus\circledomain$} \\
\cline{3-4}
&& \shortstrut $\pm\pi$ & $\Sigma_2^2$ && \\
\hline
\hline
\multirow{4}{*}{$\delta$} & \multirow{4}{*}{\makecell[l]{$(2h(z)-2h(\delta))^{2/3}$,\\
continued from $\Sigma_1^{1-}$; $h(\delta)$ \\
defined by limit along $\Sigma_1^{1-}$}} & \shortstrut $0$ & $\Sigma_1^{1-}$ &
\multirow{2}{*}{$\mathbf{T}(\sqrt{2}\ee^{\frac{5\ii\pi}{6}})$} & 
\multirow{2}{*}{$D_\delta\cap\circledomain$} \\
\cline{3-4}
&& \shortstrut $\tfrac{2}{3}\pi$ & $\Sigma_1^{2-}$ && \\
\cline{3-4}\cline{5-6}
&& \shortstrut $-\tfrac{2}{3}\pi$ & $\Sigma_1^{1+}$ \& $\Sigma_1^{2+}$ &
\multirow{2}{*}{$\mathbf{T}(\sqrt{2}\ee^{\frac{2\ii\pi}{3}})$} &
\multirow{2}{*}{$D_\delta\setminus\circledomain$} \\
\cline{3-4}
&& \shortstrut $\pm\pi$ & $\Sigma_1^2$ && \\
\hline
\end{tabular}
\end{center}
\renewcommand{\arraystretch}{1}
\end{table}

\clearpage
%\begin{figure}[h!]
%\begin{center}
%\includegraphics{SignsContourLenses-y2p025-kappa1p05-s1.pdf}
%\end{center}
%\caption{For $y_0=2.025$, $\kappa=1.05$ (representative of $y_0\in\TR$, $\kappa>1$), and $s=1$.  Top left:  Stokes graph including branch cuts (orange) for $h'(z)$ and contours $\ell_1$ and $\ell_2$.  Top right:  arcs of the jump contour $\Sigma$ for $\mathbf{M}(z)$; note that the arcs $\Sigma_1$, $\Sigma_2$, $\Sigma_\mathrm{c}$, $\Sigma_{2,3}^2$, and $\Sigma_{4,1}^2$ lie on the Stokes graph.  Below: the jump matrix $\mathbf{W}$ such that $\tilde{\mathbf{O}}_+(z)=\tilde{\mathbf{O}}_-(z)\mathbf{W}$.  In all plots: gray shading means $\mathrm{Re}(h(z))<0$ and white background means $\mathrm{Re}(h(z))>0$.}
%\label{fig:SignsContourLenses-y2p025-kappa1p05-s1}
%\end{figure}
%\clearpage
%\begin{figure}[h!]
%\begin{center}
%\includegraphics{SignsContourLenses-y2p025-kappa1p05-sm1.pdf}
%\end{center}
%\caption{As in Figure~\ref{fig:SignsContourLenses-y2p025-kappa1p05-s1} but for $s=-1$ instead.}
%\label{fig:SignsContourLenses-y2p025-kappa1p05-sm1}
%\end{figure}
%\clearpage
\subsection{The gO case with $y_0\in\TI(\kappa)$ and $s=1$}
Here we use representative values of $y_0=1.3\ii$ and $\kappa=0$.
\begin{figure}[h!]
\begin{subfigure}{\textwidth}
\begin{center}
\includegraphics[width=5 in]{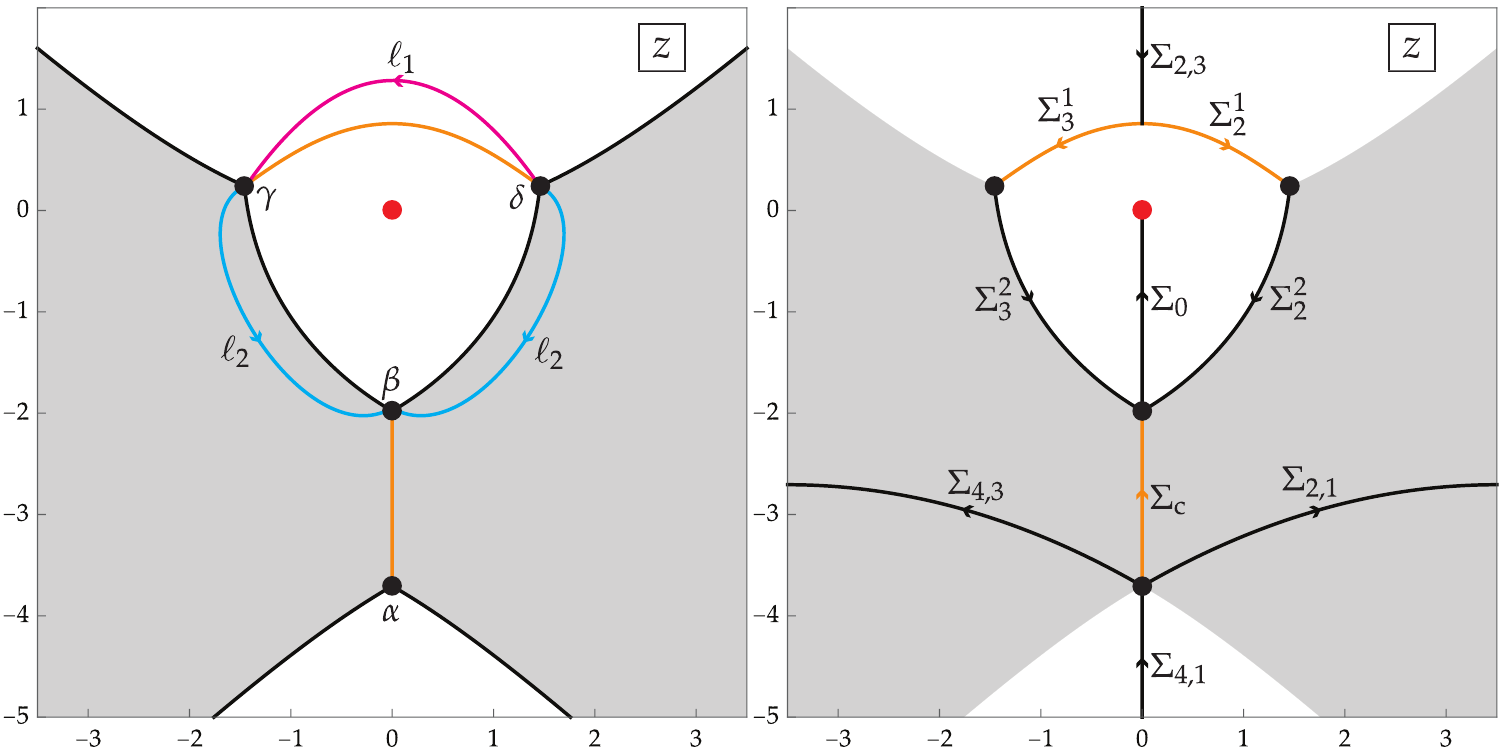}
\end{center}
\subcaption{Left:  Stokes graph including branch cuts (orange) for $h'(z)$ and contours $\ell_1$ and $\ell_2$ appearing in \eqref{eq:BoutrouxConstants}.  Right:  arcs of the jump contour $\Sigma$ for $\mathbf{M}(z)$; note that the arcs $\Sigma_j^k$, $j=2,3$ and $k=1,2$, and $\Sigma_\mathrm{c}$ lie on the Stokes graph.}
\label{fig:SignsContourLenses-y1p3i-kappa0-s1}
\end{subfigure}
\medskip
\begin{subfigure}{\textwidth}
\begin{center}
\includegraphics[width=4 in]{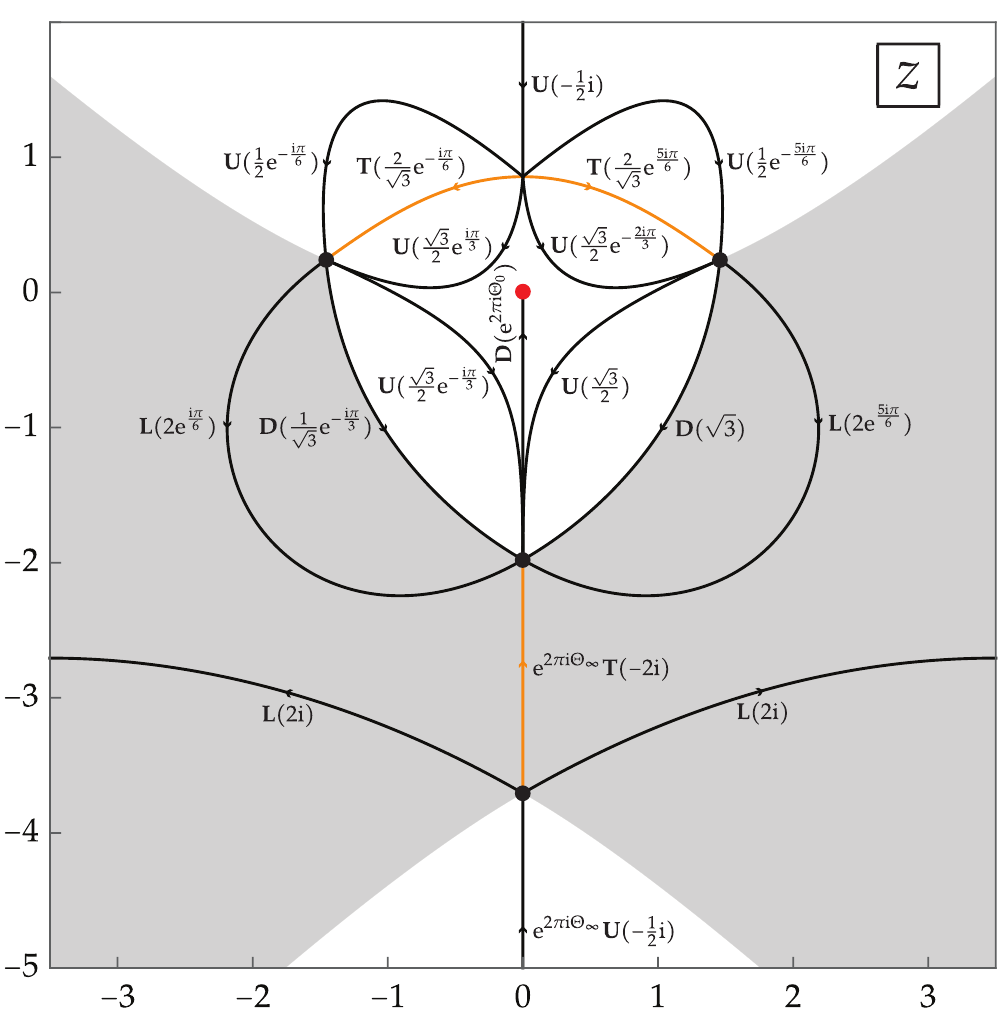}
\end{center}
\subcaption{The jump matrix $\mathbf{W}$ such that $\widetilde{\mathbf{O}}_+(z)=\widetilde{\mathbf{O}}_-(z)\mathbf{W}$.}
\label{fig:JustLenses-y1p3i-kappa0-s1}
\end{subfigure}
\caption{Diagrams for the gO case with $y_0\in\TI(\kappa)$ and $s=1$.}
\end{figure}
\begin{figure}[h]\ContinuedFloat
\begin{subfigure}{\textwidth}
\begin{center}
\includegraphics[width=5 in]{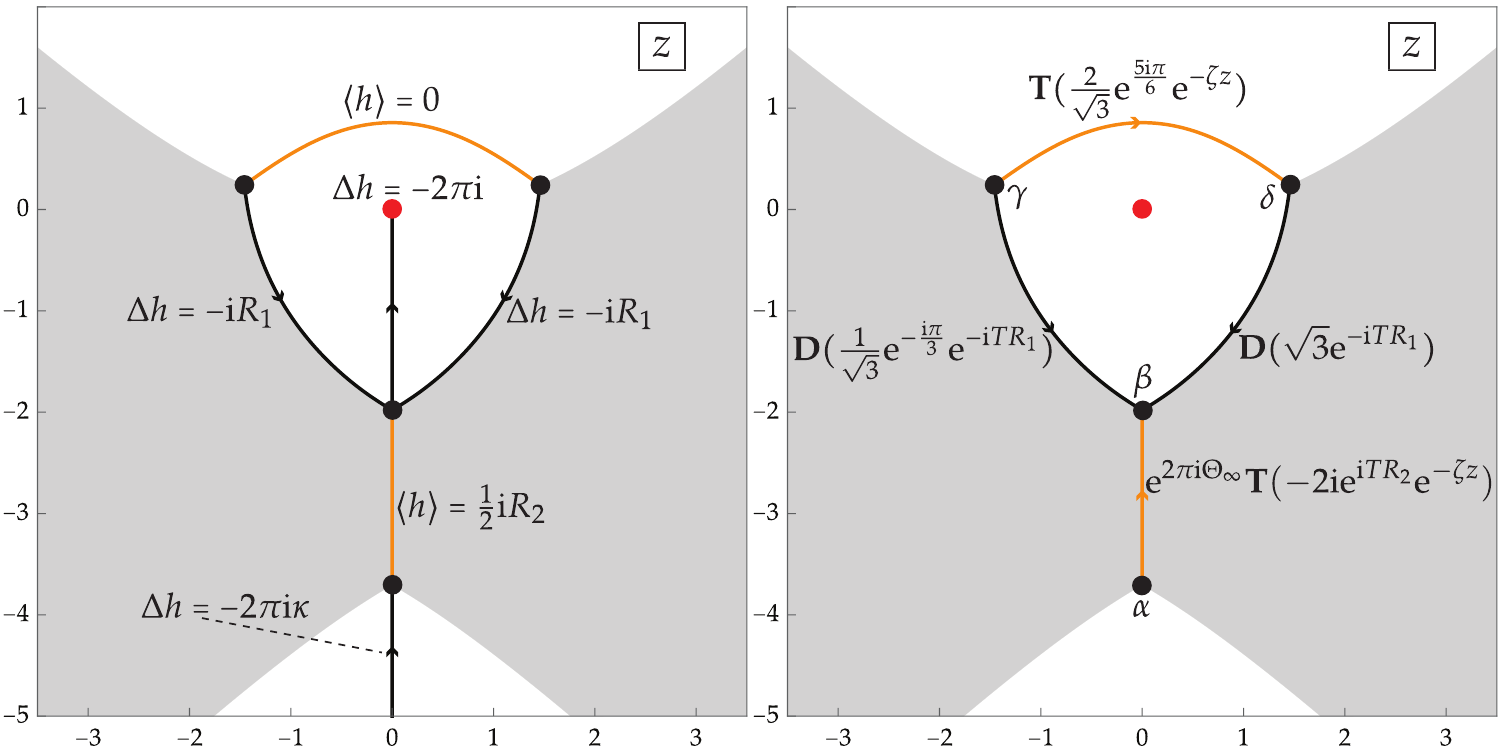}
\end{center}
\subcaption{Left:  the jump conditions satisfied by $h(z)$ with a suitable choice of integration constant.  Right:  the resulting jump conditions for the outer parametrix $\dot{\mathbf{O}}^{\mathrm{out}}(z)$.}
\label{fig:hJAO-y1p3i-kappa0-s1}
\end{subfigure}
\caption{Diagrams for the gO case with $y_0\in\TI(\kappa)$ and $s=1$ (contd.)}
\end{figure}
\begin{table}
\caption{Inner parametrix data for the gO case with $y_0\in\TI(\kappa)$ and $s=1$.}
\renewcommand{\arraystretch}{0.45}
\begin{center}
\begin{tabular}{@{}|l|l|c|c|c|c|@{}}
\hline
\multirow{2}{*}{$p$}&\multirow{2}{*}{Conformal map $W:D_p\to\mathbb{C}$}&\multicolumn{2}{c|}{Ray Preimages in $D_p$} & \multicolumn{2}{c|}{$\mathbf{C}(z)$ in $D_p$}\\
\cline{3-6}
& & $\arg(W)$ & Preimage & Value $\mathbf{C}$ & Subdomain of $D_p$ \\
\hline\hline
\multirow{4}{*}{$\alpha$} & \multirow{4}{*}{\makecell[l]{$(2\langle h\rangle(z)-2\langle h\rangle(\alpha))^{2/3}$,\\ 
continued from $\Sigma_{4,1}$
%; $\langle h\rangle(\alpha)$\\ defined by limit along $\Sigma_{4,1}$
}} &
\shortstrut $0$ & $\Sigma_{4,1}$ & 
\multirow{2}{*}{$\mathbf{T}(\sqrt{2}\ii)$} & 
\multirow{2}{*}{$D_\alpha$, left of $\Sigma_{4,1}$ \& $\Sigma_\mathrm{c}$} \\
\cline{3-4}
&& \shortstrut $\tfrac{2}{3}\pi$ & $\Sigma_{2,1}$ && \\
\cline{3-4}\cline{5-6}
&& \shortstrut $-\tfrac{2}{3}\pi$ & $\Sigma_{4,3}$ & 
\multirow{2}{*}{$\ee^{2\pi\ii\Theta_\infty}\mathbf{T}(\sqrt{2}\ii)$} & 
\multirow{2}{*}{$D_\alpha$, right of $\Sigma_{4,1}$ \& $\Sigma_\mathrm{c}$} \\
\cline{3-4}
&& \shortstrut $\pm\pi$ & $\Sigma_\mathrm{c}$ && \\
\hline
\hline
\multirow{4}{*}{$\beta$} & \multirow{4}{*}{\makecell[l]{$(2\langle h\rangle(z)-2\langle h\rangle(\beta))^{2/3}$,\\
continued from $\Sigma_0$; $\langle h\rangle(\beta)$\\
defined by limit along $\Sigma_0$}} & \shortstrut $0$ & $\Sigma_2^{2-}$, $\Sigma_0$, \& $\Sigma_3^{2+}$ &
$\mathbf{T}(\sqrt{\tfrac{2}{3}}\ii)$ & $D_\beta\cap\circledomain$, left of $\Sigma_0$ \\
\cline{3-4}\cline{5-6}
&& \shortstrut $\tfrac{2}{3}\pi$ & $\Sigma_3^{2-}$ & $\mathbf{T}(\sqrt{\tfrac{2}{3}}\ii\ee^{-2\pi\ii\Theta_0})$ & $D_\beta\cap\circledomain$, right of $\Sigma_0$ \\
\cline{3-4}\cline{5-6}
&& \shortstrut $-\tfrac{2}{3}\pi$ & $\Sigma_2^{2+}$ & $\mathbf{T}(\sqrt{2}\ee^{\frac{5\ii\pi}{6}})$ & $D_\beta\setminus\circledomain$, left of $\Sigma_\mathrm{c}$ \\
\cline{3-4}\cline{5-6}
&& \shortstrut $\pm\pi$ & $\Sigma_\mathrm{c}$ & $\mathbf{T}(\sqrt{2}\ii\ee^{-2\pi\ii\Theta_0})$ &
$D_\beta\setminus\circledomain$, right of $\Sigma_\mathrm{c}$ \\
\hline
\hline
\multirow{4}{*}{$\gamma$} & \multirow{4}{*}{\makecell[l]{$(2h(z)-2h(\gamma))^{2/3}$,\\
continued from $\Sigma_3^{1-}$; $h(\gamma)$ \\
defined by limit along $\Sigma_3^{1-}$}} & \shortstrut $0$ & $\Sigma_3^{1-}$ &
\multirow{2}{*}{$\mathbf{D}(\sqrt{\tfrac{3}{2}}\ii)$} & 
\multirow{2}{*}{$D_\gamma\cap\circledomain$} \\
\cline{3-4}
&& \shortstrut $\tfrac{2}{3}\pi$ & $\Sigma_3^{2-}$ && \\
\cline{3-4}\cline{5-6}
&& \shortstrut $-\tfrac{2}{3}\pi$ & $\Sigma_3^{1+}$ \& $\Sigma_3^{2+}$ & 
\multirow{2}{*}{$\mathbf{T}(\sqrt{2}\ee^{\frac{\ii\pi}{3}})$} & 
\multirow{2}{*}{$D_\gamma\setminus\circledomain$} \\
\cline{3-4}
&& \shortstrut $\pm\pi$ & $\Sigma_3^2$ && \\
\hline
\hline
\multirow{4}{*}{$\delta$} & \multirow{4}{*}{\makecell[l]{$(2h(z)-2h(\delta))^{2/3}$,\\
continued from $\Sigma_2^{1+}$; $h(\delta)$ \\
defined by limit along $\Sigma_2^{1+}$}} & \shortstrut $0$ & $\Sigma_2^{1+}$ &
\multirow{2}{*}{$\mathbf{D}(\sqrt{\frac{3}{2}}\ee^{\frac{5\ii\pi}{6}})$} & 
\multirow{2}{*}{$D_\delta\cap\circledomain$} \\
\cline{3-4}
&& \shortstrut $\tfrac{2}{3}\pi$ & $\Sigma_2^{1-}$ \& $\Sigma_2^{2-}$ && \\
\cline{3-4}\cline{5-6}
&& \shortstrut $-\tfrac{2}{3}\pi$ & $\Sigma_2^{2+}$ &
\multirow{2}{*}{$\mathbf{T}(\sqrt{2}\ee^{\frac{2\ii\pi}{3}})$} &
\multirow{2}{*}{$D_\delta\setminus\circledomain$} \\
\cline{3-4}
&& \shortstrut $\pm\pi$ & $\Sigma_2^2$ && \\
\hline
\end{tabular}
\end{center}
\renewcommand{\arraystretch}{1}
\label{tab:Inner-y1p3i-kappa0-s1}
\end{table}
\clearpage

\subsection{The gO case with $y_0\in\TI(\kappa)$ and $s=-1$}
Here we use representative values of $y_0=1.3\ii$ and $\kappa=0$.
\begin{figure}[h!]
\begin{subfigure}{\textwidth}
\begin{center}
\includegraphics[width=5 in]{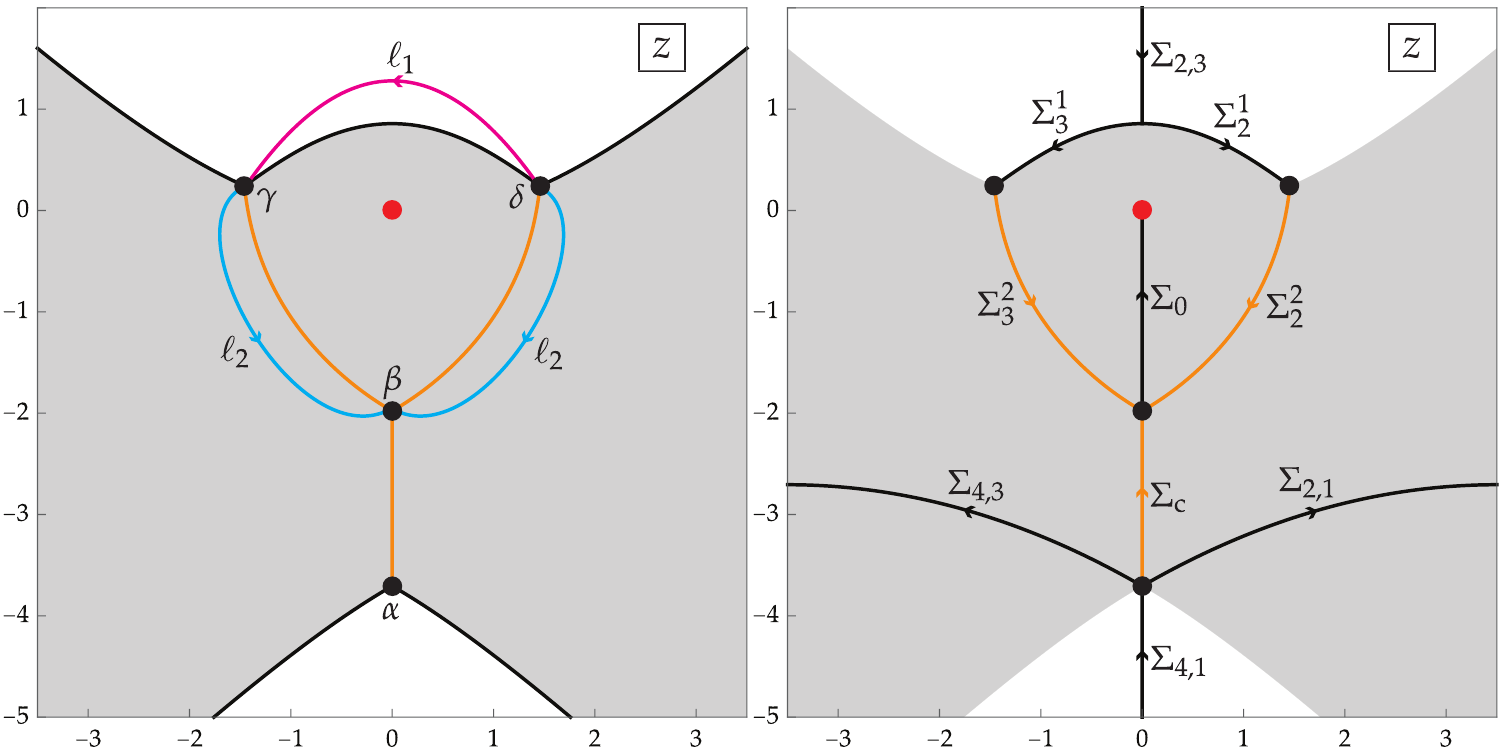}
\end{center}
\subcaption{Left:  Stokes graph including branch cuts (orange) for $h'(z)$ and contours $\ell_1$ and $\ell_2$ appearing in \eqref{eq:BoutrouxConstants}.  Right:  arcs of the jump contour $\Sigma$ for $\mathbf{M}(z)$; note that the arcs $\Sigma_j^k$, $j=2,3$ and $k=1,2$, and $\Sigma_\mathrm{c}$ lie on the Stokes graph.}
\label{fig:SignsContourLenses-y1p3i-kappa0-sm1}
\end{subfigure}
\medskip
\begin{subfigure}{\textwidth}
\begin{center}
\includegraphics[width=4 in]{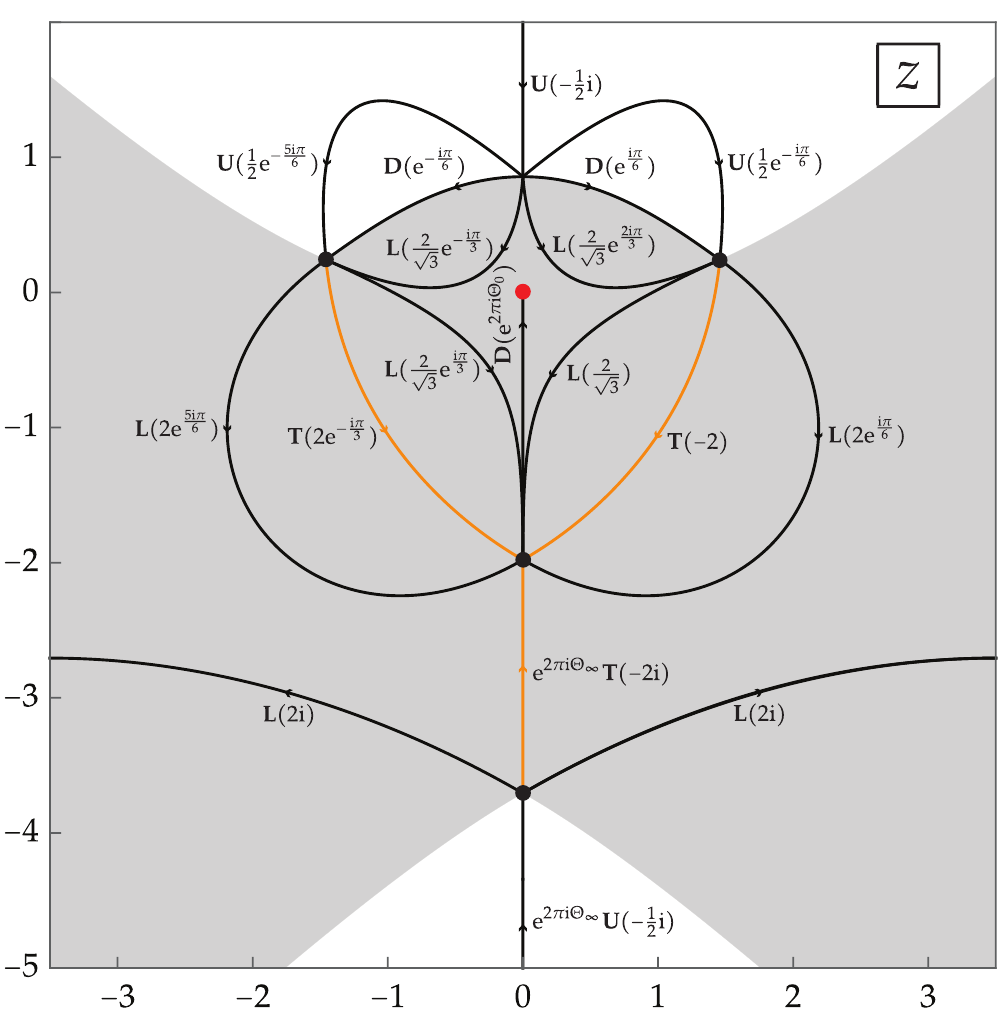}
\end{center}
\subcaption{The jump matrix $\mathbf{W}$ such that $\widetilde{\mathbf{O}}_+(z)=\widetilde{\mathbf{O}}_-(z)\mathbf{W}$.}
\label{fig:JustLenses-y1p3i-kappa0-sm1}
\end{subfigure}
\caption{Diagrams for the gO case with $y_0\in\TI(\kappa)$ and $s=-1$.}
\end{figure}
\begin{figure}[h]\ContinuedFloat
\begin{subfigure}{\textwidth}
\begin{center}
\includegraphics[width=5 in]{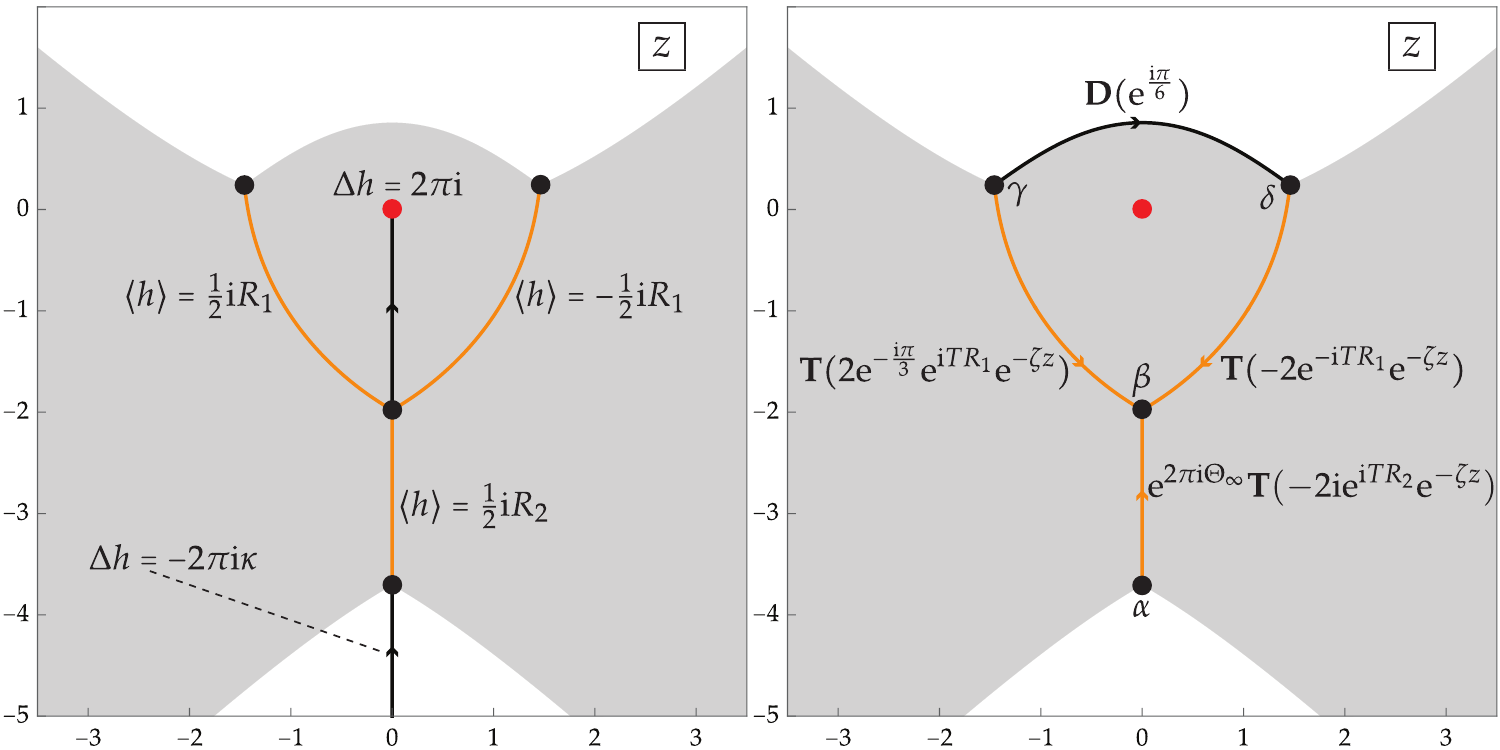}
\end{center}
\subcaption{Left:  the jump conditions satisfied by $h(z)$ with a suitable choice of integration constant.  Right:  the resulting jump conditions for the outer parametrix $\dot{\mathbf{O}}^{\mathrm{out}}(z)$.}
\label{fig:hJAO-y1p3i-kappa0-sm1}
\end{subfigure}
\caption{Diagrams for the gO case with $y_0\in\TI(\kappa)$ and $s=-1$ (contd.)}
\end{figure}
\begin{table}
\caption{Inner parametrix data for the gO case with $y_0\in\TI(\kappa)$ and $s=-1$.}
\renewcommand{\arraystretch}{0.45}
\begin{center}
\begin{tabular}{@{}|l|l|c|c|c|c|@{}}
\hline
\multirow{2}{*}{$p$}&\multirow{2}{*}{Conformal map $W:D_p\to\mathbb{C}$}&\multicolumn{2}{c|}{Ray Preimages in $D_p$} & \multicolumn{2}{c|}{$\mathbf{C}(z)$ in $D_p$}\\
\cline{3-6}
& & $\arg(W)$ & Preimage & Value $\mathbf{C}$ & Subdomain of $D_p$ \\
\hline\hline
\multirow{4}{*}{$\alpha$} & \multirow{4}{*}{\makecell[l]{$(2\langle h\rangle(z)-2\langle h\rangle(\alpha))^{2/3}$,\\ 
continued from $\Sigma_{4,1}$
%; $\langle h\rangle(\alpha)$\\ defined by limit along $\Sigma_{4,1}$
}} &
\shortstrut $0$ & $\Sigma_{4,1}$ & 
\multirow{2}{*}{$\mathbf{T}(\sqrt{2}\ii)$} & 
\multirow{2}{*}{$D_\alpha$, left of $\Sigma_{4,1}$ \& $\Sigma_\mathrm{c}$} \\
\cline{3-4}
&& \shortstrut $\tfrac{2}{3}\pi$ & $\Sigma_{2,1}$ && \\
\cline{3-4}\cline{5-6}
&& \shortstrut $-\tfrac{2}{3}\pi$ & $\Sigma_{4,3}$ & 
\multirow{2}{*}{$\ee^{2\pi\ii\Theta_\infty}\mathbf{T}(\sqrt{2}\ii)$} & 
\multirow{2}{*}{$D_\alpha$, right of $\Sigma_{4,1}$ \& $\Sigma_\mathrm{c}$} \\
\cline{3-4}
&& \shortstrut $\pm\pi$ & $\Sigma_\mathrm{c}$ && \\
\hline
\hline
\multirow{4}{*}{$\beta$} & \multirow{4}{*}{\makecell[l]{$(2\langle h\rangle(\beta)-2\langle h\rangle(z))^{2/3}$,\\
continued from $\Sigma_0$; $\langle h\rangle(\beta)$\\
defined by limit along $\Sigma_0$}} & \shortstrut $0$ & $\Sigma_2^{2-}$, $\Sigma_0$, \& $\Sigma_3^{2+}$ &
$\mathbf{D}(\tfrac{1}{\sqrt{2}}\ii)$ & $D_\beta\cap\circledomain$, left of $\Sigma_0$ \\
\cline{3-4}\cline{5-6}
&& \shortstrut $\tfrac{2}{3}\pi$ & $\Sigma_3^{2-}$ & $\mathbf{D}(\tfrac{1}{\sqrt{2}}\ii\ee^{2\pi\ii\Theta_0})$ & $D_\beta\cap\circledomain$, right of $\Sigma_0$ \\
\cline{3-4}\cline{5-6}
&& \shortstrut $-\tfrac{2}{3}\pi$ & $\Sigma_2^{2+}$ & $\mathbf{T}(\sqrt{2}\ee^{\frac{\ii\pi}{6}})$ & $D_\beta\setminus\circledomain$, left of $\Sigma_\mathrm{c}$ \\
\cline{3-4}\cline{5-6}
&& \shortstrut $\pm\pi$ & $\Sigma_\mathrm{c}$ & $\mathbf{T}(\sqrt{2}\ii\ee^{2\pi\ii\Theta_0})$ &
$D_\beta\setminus\circledomain$, right of $\Sigma_\mathrm{c}$ \\
\hline
\hline
\multirow{4}{*}{$\gamma$} & \multirow{4}{*}{\makecell[l]{$(2h(z)-2h(\gamma))^{2/3}$,\\
continued from $\Sigma_3^{1-}$
%; $h(\gamma)$ \\ defined by limit along $\Sigma_3^{1-}$
}} & \shortstrut $0$ & $\Sigma_3^{1-}$ &
\multirow{2}{*}{$\mathbf{T}(\sqrt{2}\ii)$} & 
\multirow{2}{*}{$D_\gamma\cap\circledomain$} \\
\cline{3-4}
&& \shortstrut $\tfrac{2}{3}\pi$ & $\Sigma_3^{2-}$ && \\
\cline{3-4}\cline{5-6}
&& \shortstrut $-\tfrac{2}{3}\pi$ & $\Sigma_3^{1+}$ \& $\Sigma_3^{2+}$ & 
\multirow{2}{*}{$\mathbf{T}(\sqrt{2}\ee^{\frac{2\ii\pi}{3}})$} & 
\multirow{2}{*}{$D_\gamma\setminus\circledomain$} \\
\cline{3-4}
&& \shortstrut $\pm\pi$ & $\Sigma_3^2$ && \\
\hline
\hline
\multirow{4}{*}{$\delta$} & \multirow{4}{*}{\makecell[l]{$(2h(z)-2h(\delta))^{2/3}$,\\
continued from $\Sigma_2^{1+}$
%; $h(\delta)$ \\ defined by limit along $\Sigma_2^{1+}$
}} & \shortstrut $0$ & $\Sigma_2^{1+}$ &
\multirow{2}{*}{$\mathbf{T}(\sqrt{2}\ee^{\frac{\ii\pi}{6}})$} & 
\multirow{2}{*}{$D_\delta\cap\circledomain$} \\
\cline{3-4}
&& \shortstrut $\tfrac{2}{3}\pi$ & $\Sigma_2^{1-}$ \& $\Sigma_2^{2-}$ && \\
\cline{3-4}\cline{5-6}
&& \shortstrut $-\tfrac{2}{3}\pi$ & $\Sigma_2^{2+}$ &
\multirow{2}{*}{$\mathbf{T}(\sqrt{2}\ee^{\frac{\ii\pi}{3}})$} &
\multirow{2}{*}{$D_\delta\setminus\circledomain$} \\
\cline{3-4}
&& \shortstrut $\pm\pi$ & $\Sigma_2^2$ && \\
\hline
\end{tabular}
\end{center}
\renewcommand{\arraystretch}{1}
\label{tab:Inner-y1p3i-kappa0-sm1}
\end{table}
\clearpage

\section{Effective approximation of rational solutions on Boutroux domains}
\label{app:Effective}
The approximation most directly adapted to our analysis of Riemann-Hilbert Problem~\ref{rhp:general} is that of the rational functions $u^{[3]}_\mathrm{gO}(x;m,n)$ and $u^{[3]}_\mathrm{gH}(x;m,n)$.  The basic approximation formula for these functions reads:
\begin{multline}
u_\mathrm{F}^{[3]}\left(|\Theta_{0,\mathrm{F}}^{[3]}(m,n)|^\frac{1}{2}y_0+|\Theta_{0,\mathrm{F}}^{[3]}(m,n)|^{-\frac{1}{2}}\zeta;m,n\right)^\chi=|\Theta_{0,\mathrm{F}}^{[3]}(m,n)|^{\frac{1}{2}\chi}\left(f(\zeta-\zeta_0)^\chi +O\left(|\Theta_{0,\mathrm{F}}^{[3]}(m,n)|^{-1}\right)\right),\\
f(\zeta-\zeta_0)=\dot{U}^{[3]}_\mathrm{F}(\zeta;y_0):=\psi_\mathrm{F}^{[3]}(y_0)\frac{\Theta(K-\ii(\varphi-\mathfrak{z}_1^{[3]}))\Theta(K-\ii(\varphi-\mathfrak{z}_2^{[3]}))}
{\Theta(K-\ii(\varphi-\mathfrak{p}_1^{[3]}))\Theta(K-\ii(\varphi-\mathfrak{p}_2^{[3]}))}.
\label{eq:Type3-summary}
\end{multline}
Assuming that $\chi=-\mathrm{sgn}(\ln |f(\zeta-\zeta_0)|)$, 
the error term is uniform for bounded $\zeta$ and for $y_0$ in a compact subset of the selected Boutroux domain $\mathcal{B}$.
% provided that $(y_0,\zeta,|\Theta_{0,\mathrm{F}}^{[3]}(m,n)|)\in\mathcal{S}(\epsilon)$ for some $\epsilon>0$ (equivalent to $\zeta-\zeta_0$ bounded away from $\mathcal{M}^{[3]}$). 
Computing this leading term consists of the following steps. 
\begin{enumerate}
\item Define the parameters $T>0$, $s=\pm 1$, and $\kappa\in (-1,1)$ in terms of $(m,n)$ by \eqref{eq:gH-type3-RHP-parameters} (resp., by \eqref{eq:gO-type3-RHP-parameters}) for the gH family (resp., for the gO family).
\item Select a Boutroux domain $\mathcal{B}=\TR(\kappa)$, $\mathcal{B}=\TI(\kappa)$ (both for the gO family only), or $\mathcal{B}=\rectangle(\kappa)$, and ensure that $y_0\in\mathcal{B}$ (the boundaries of the domains can be numerically computed given $\kappa$ using \eqref{eq:intro-arcs}).
\item Using numerical root finding and continuation methods informed by the relevent section among Sections~\ref{sec:rectangle-domain}, \ref{sec:TR}, and \ref{sec:TI}, determine the value of $E=E(y_0;\kappa)$ for which the Boutroux equations \eqref{eq:Boutroux} hold.  
\item With $E$ determined, the polynomial $P(z)$ given by \eqref{eq:elliptic-ODE} is now known.  Find its roots $\alpha$, $\beta$, $\gamma$, and $\delta$ and order them according to the Stokes graph as illustrated in Figures~\ref{fig:SignsContourLenses-y0-kappa0-s1}--\ref{fig:SignsContourLenses-y1p3i-kappa0-sm1} relevant for the family, Boutroux domain, and sign $s$ under consideration (recall that while these figures are for special cases of the parameters, the abstract Stokes graph depends only on the selected Boutroux domain).  Then using the contours $\ell_1$ and $\ell_2$ and branch cuts for $R(z)$ illustrated in the same plots, numerically compute the real constants $R_1$ and $R_2$ given by \eqref{eq:BoutrouxConstants} (one can use this computation as an opportunity to verify the accuracy of the determination of $E$, which should force the imaginary parts of $R_1$ and $R_2$ to vanish to machine precision).
\item Matching the topological representation of $\mathfrak{a}$ and $\mathfrak{b}$ cycles shown in Figure~\ref{fig:OuterParametrixContours} with the actual ordering of the points obtained in the previous step, and using the relationship \eqref{eq:r-to-R}, numerically calculate the constant $c$ given in \eqref{eq:omega} and the constant $H_\omega$ given in \eqref{B-cycle}.  Then taking into account the value of $z_0$ given in terms of the well-defined branch points $\alpha$, $\beta$, $\gamma$, $\delta$ by \eqref{eq:z0-define}, numerically evaluate the integrals $a(0)$, $a(\infty)$, and $a(z_0)$ (see \eqref{Abel-map}).  Using these, define the phase shifts $\mathfrak{z}_j^{\smash{[3]}}$ and $\mathfrak{p}^{\smash{[3]}}_j$, $j=1,2$, by \eqref{eq:dotU-phases}.
\item
Using Table~\ref{tab:outer-uniformize-phases} to determine the phases $C_\mathrm{G}$ and $C_\mathrm{B}$ relevant to the case at hand (noting also $\Theta_\infty=-\kappa T$), define the aggregate phase $\varphi$ by \eqref{eq:F1-U}.
\item 
Recalling the definition \eqref{theta-function} of $\Theta(z)$ given $H_\omega$ (possibly implementing this definition using built-in function calls to Jacobi theta functions) and using $K=\ii\pi +\tfrac{1}{2}H_\omega$, calculate the complex amplitude factor $\psi^{[3]}_\mathrm{F}(y_0)$ from \eqref{eq:psi-simpler}.  Then bring in $\varphi$ and the four phase shifts to finish the rest of the calculation using \eqref{eq:Type3-summary}.
\end{enumerate}
The other approximation proved by analyzing the same Riemann-Hilbert Problem applies to the rational functions $u_\mathrm{gO}^{[1]}(x;m,n)$ and $u_\mathrm{gH}^{[1]}(x;m,n)$, and it reads
\begin{multline}
u_\mathrm{F}^{[1]}\left(|\Theta_{0,\mathrm{F}}^{[1]}(m,n)|^\frac{1}{2}\widehat{y_0}+|\Theta_{0,\mathrm{F}}^{[1]}(m,n)|^{-\frac{1}{2}}\widehat{\zeta};m,n\right)^\chi=|\Theta_{0,\mathrm{F}}^{[1]}(m,n)|^{\frac{1}{2}\chi}\left(f(\zeta-\zeta_0)^\chi +O\left(|\Theta_{0,\mathrm{F}}^{[1]}(m,n)|^{-1}\right)\right),\\
f(\widehat{\zeta}-\widehat{\zeta_0})=\dot{U}^{[1]}_\mathrm{F}(\widehat{\zeta};\widehat{y_0}):=\psi_\mathrm{F}^{[1]}(\widehat{y_0})\frac{\Theta(K-\ii(\varphi-\mathfrak{z}_1^{[1]}))\Theta(K-\ii(\varphi-\mathfrak{z}_2^{[1]}))}
{\Theta(K-\ii(\varphi-\mathfrak{p}_1^{[1]}))\Theta(K-\ii(\varphi-\mathfrak{p}_2^{[1]}))}.
\label{eq:Type1-summary}
\end{multline}
Again taking $\chi=-\mathrm{sgn}(\ln|f(\widehat{\zeta}-\widehat{\zeta_0})|)$,
the error term is uniform for bounded $\widehat{\zeta}$ and for $\widehat{y_0}$ in a compact subset of the chosen Boutroux domain $\mathcal{B}$.
% if also $\widehat{\zeta}-\widehat{\zeta_0}$ is bounded away from $\mathcal{M}^{[1]}$.  
To compute this leading term, we modify the above steps as follows.
\begin{enumerate}
\item Define the parameters $T>0$, $s=\pm 1$, and $\kappa\in(-1,1)$ in terms of $(m,n)$ by \eqref{eq:gH-type1-RHP-parameters} (resp., by \eqref{eq:gO-type1-RHP-parameters}) for the gH family (resp., for the gO family).  These are not directly related to the ``native'' parameters $\Theta_{0,\mathrm{F}}^{[1]}(m,n)$ and $\Theta_{\infty,\mathrm{F}}^{[1]}(m,n)$, but they are the correct values to use for the remaining steps of the calculation.  From the given values of the variables $\widehat{\zeta}$ and $\widehat{y_0}$, define scaled versions needed for the subsequent steps by setting  
\eq
\zeta:=\sqrt{\frac{2}{1-s\kappa}}\widehat{\zeta}\quad\text{and}\quad y_0:=\sqrt{\frac{1-s\kappa}{2}}\widehat{y_0}.
\endeq
\item 
As above.
\item
As above.
\item
As above.
\item
As above, but instead calculate the phase shifts $\mathfrak{z}_j^{\smash{[1]}}$ and $\mathfrak{p}_j^{\smash{[1]}}$ from \eqref{eq:dotUtwist-phases}.
\item
As above for the indicated parameters.
\item
As above for the indicated parameters, but now define the complex amplitude $\psi_\mathrm{F}^{[1]}(\widehat{y_0})=\ee^{\ii(\mathfrak{z}_1^{[1]}-\mathfrak{p}_2^{[1]})}M$ in terms of $M$ given in \eqref{eq:psi-twist-simpler}, and finish the calculation using \eqref{eq:Type1-summary}.  
\end{enumerate}
We do not obtain any approximations for the rational solutions $u_\mathrm{gO}^{[2]}(x;m,n)$ or $u_\mathrm{gH}^{[2]}(x;m,n)$ directly from analysis of Riemann-Hilbert Problem~\ref{rhp:general}, but we can apply the exact symmetry \eqref{eq:symmetry-1-2} to obtain the following result.
\begin{multline}
u_\mathrm{F}^{[2]}\left(|\Theta_{0,\mathrm{F}}^{[2]}(m,n)|^\frac{1}{2}y_0+|\Theta_{0,\mathrm{F}}^{[2]}(m,n)|^{-\frac{1}{2}}\zeta;m,n\right)^\chi=|\Theta_{0,\mathrm{F}}^{[2]}(m,n)|^{\frac{1}{2}\chi}\left(f(\zeta-\zeta_0)^\chi +O\left(|\Theta_{0,\mathrm{F}}^{[2]}(m,n)|^{-1}\right)\right),\\
f(\zeta-\zeta_0)=\dot{U}^{[2]}_\mathrm{F}(\zeta;y_0):=\ii\dot{U}^{[1]}_\mathrm{F}(-\ii\zeta;-\ii y_0).
\end{multline}
Taking $\chi=-\mathrm{sgn}(\ln|f(\zeta-\zeta_0)|)$, this formula is also uniformly valid for bounded $\zeta$ and $y_0$ in a compact subset of the chosen Boutroux domain $\mathcal{B}$, provided that $\zeta-\zeta_0$ is bounded away from $\mathcal{M}^{[2]}$.  The leading term can obviously be computed by adapting the above procedure to variables rotated by $-\ii$ in the complex plane.

\section{Alternate approach to $u^{[1]}_\mathrm{gO}(x;m,n)$ and $u^{[1]}_\mathrm{gH}(x;m,n)$}
\label{app:Alternate}
The basic approach we have followed in this paper is to use the fact that the Painlev\'e-IV rational solutions of type $1$, which correspond to values of $\kappa$ outside of the basic interval $(-1,1)$, can be extracted via the formula \eqref{eq:u-ucirc} for $u_\tw(x)$ from Riemann-Hilbert Problem~\ref{rhp:general} formulated for the rational solutions of type $3$, which correspond instead to $\kappa\in (-1,1)$.  Of course another approach to the rational solutions of type $1$ (and also type $2$) is to represent these solutions as $u(x)$ instead of $u_\tw(x)$ in \eqref{eq:u-ucirc} and solve Riemann-Hilbert Problem~\ref{rhp:general} for parameters covering all lattice points far from the origin in Figure~\ref{fig:ThetasPlane}.  The latter approach avoids the complication of the changes of variables associated with the B\"acklund transformation $u(x)\mapsto u_\tw(x)$, but it leads to many additional cases for matrix factorizations and parametrix constructions, as one must consider spectral curves for $\kappa<-1$ and $\kappa>1$, as well as for $\kappa\in(-1,1)$.  

Just to give a flavor of the differences that can arise for $|\kappa|>1$, we present here the analogues of Figures~\ref{fig:Trajectories-gO-k0} and \ref{fig:Trajectories-gH-k0} in which we display the critical v-trajectories of the quadratic differential $h'(z)^2\,\dd z^2$ emanating from (generically) simple roots of the quartic $P(z)$.
\begin{figure}[h]
\hspace{-0.2in}
\begin{tabular}{c}
\includegraphics[height=1.2in]{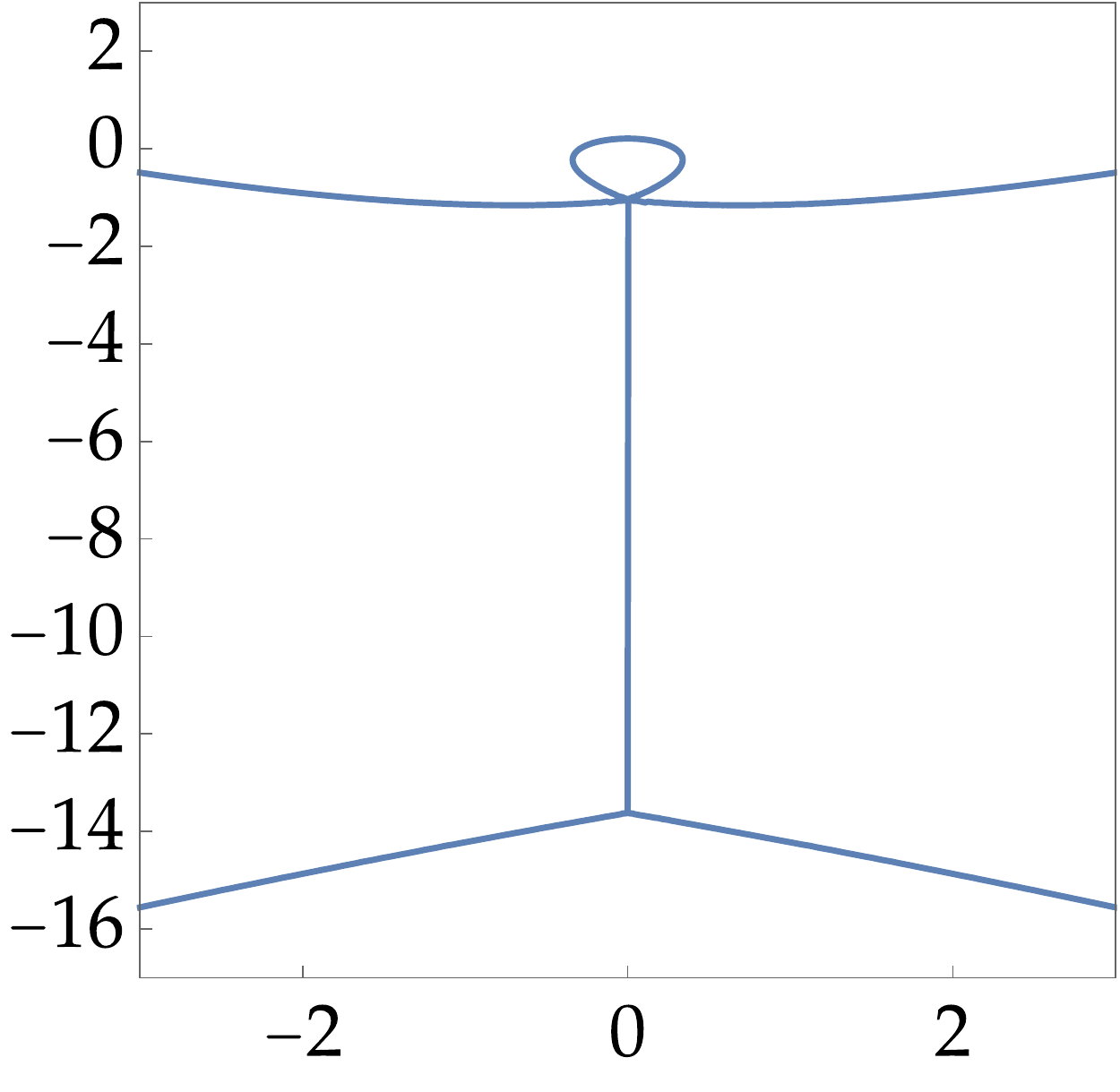}\\
\includegraphics[height=1.2in]{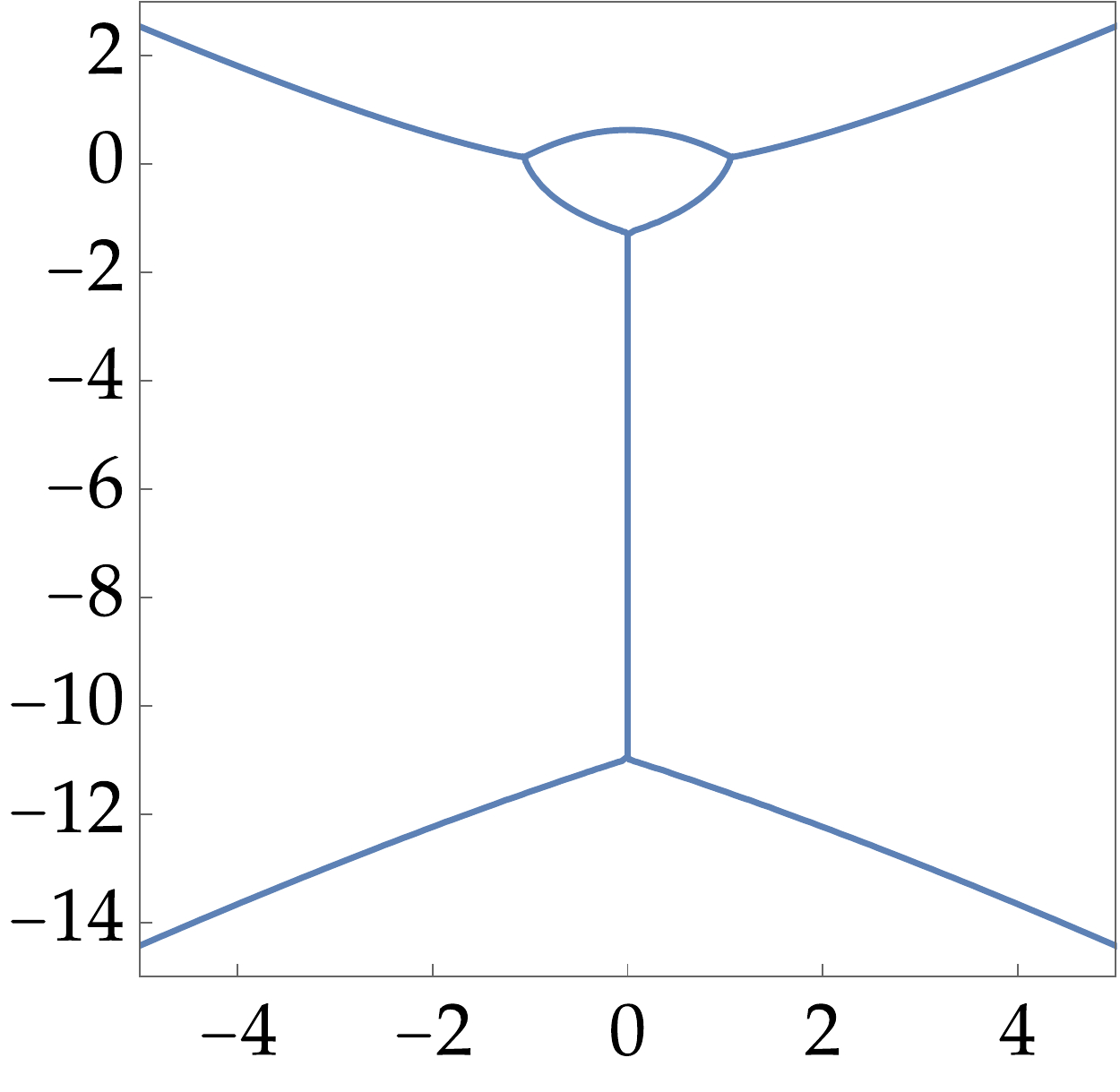}\\
\includegraphics[height=1.2in]{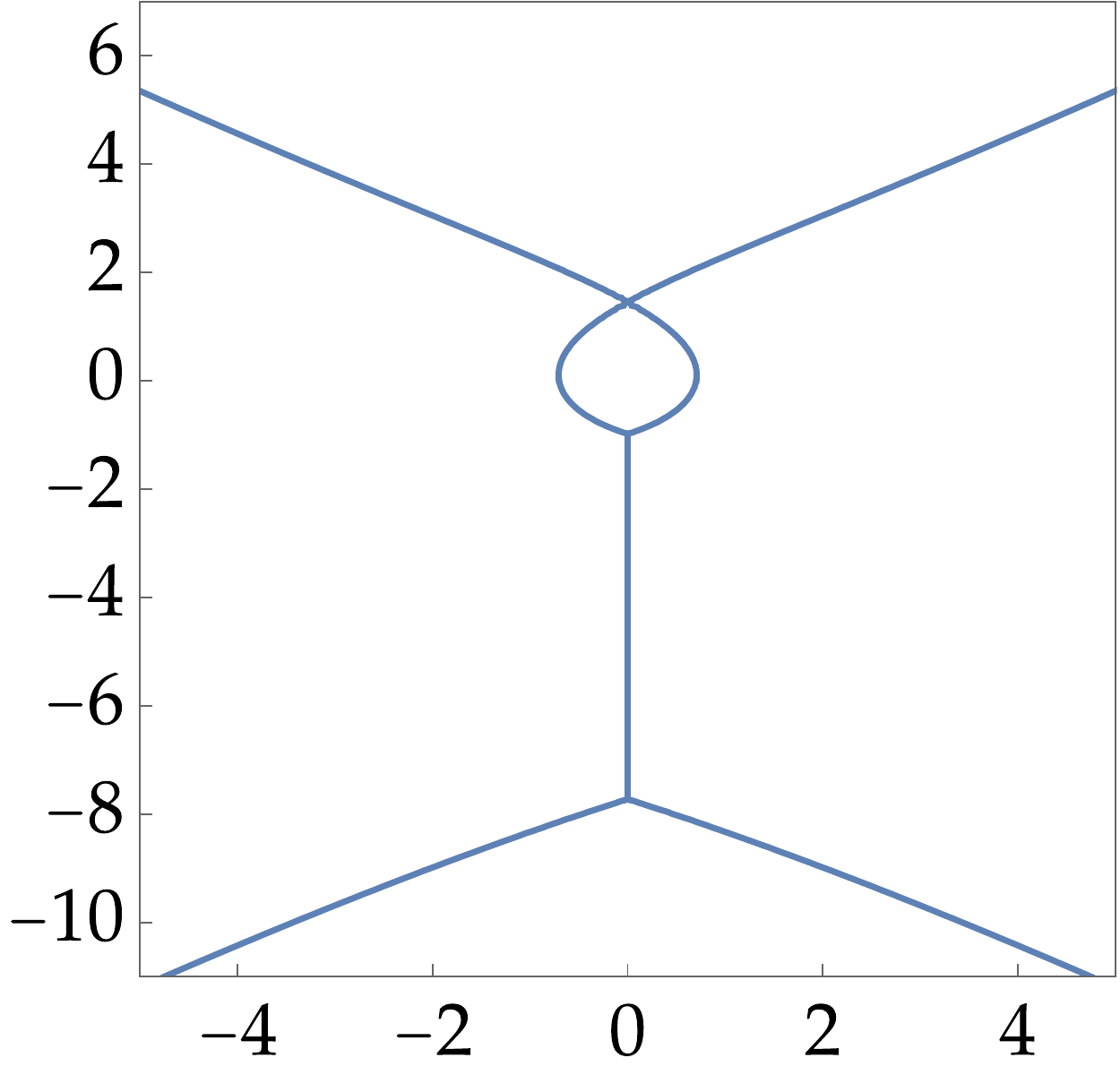}\\
\includegraphics[height=1.2in]{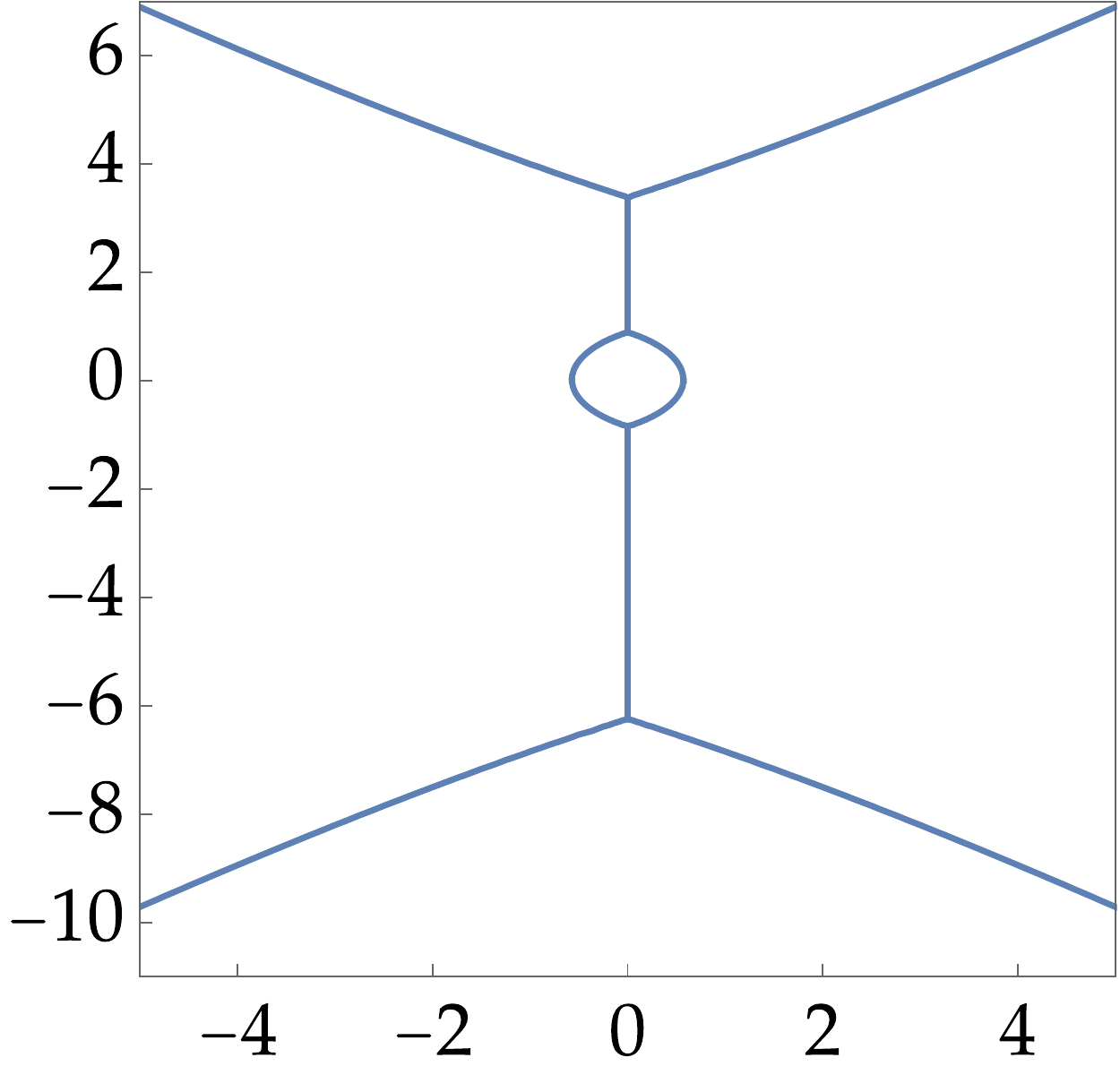}\\
\includegraphics[height=1.2in]{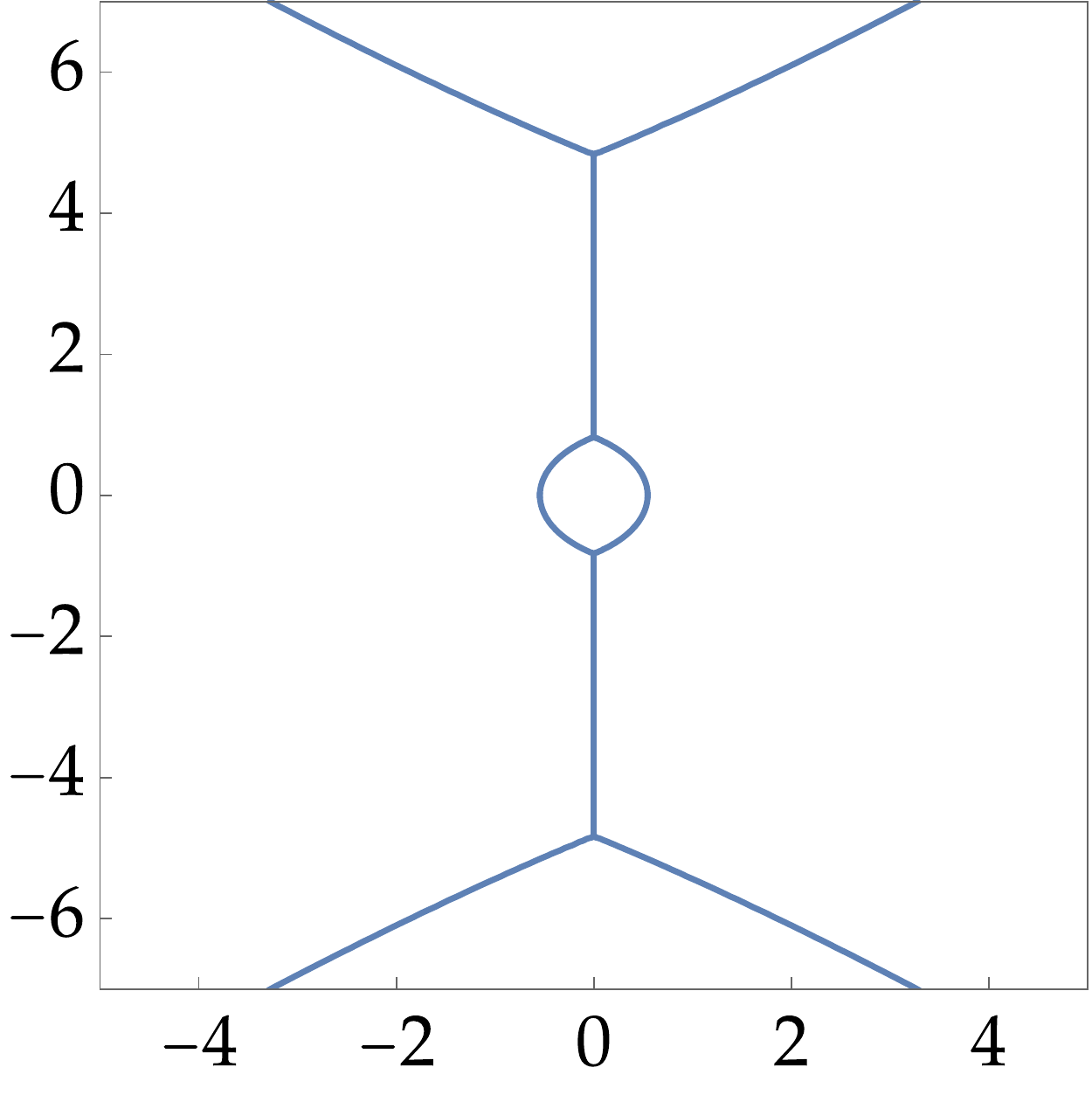}\\
\end{tabular}
\hspace{-.18in}
\begin{tabular}{c}
\includegraphics[height=1.2in]{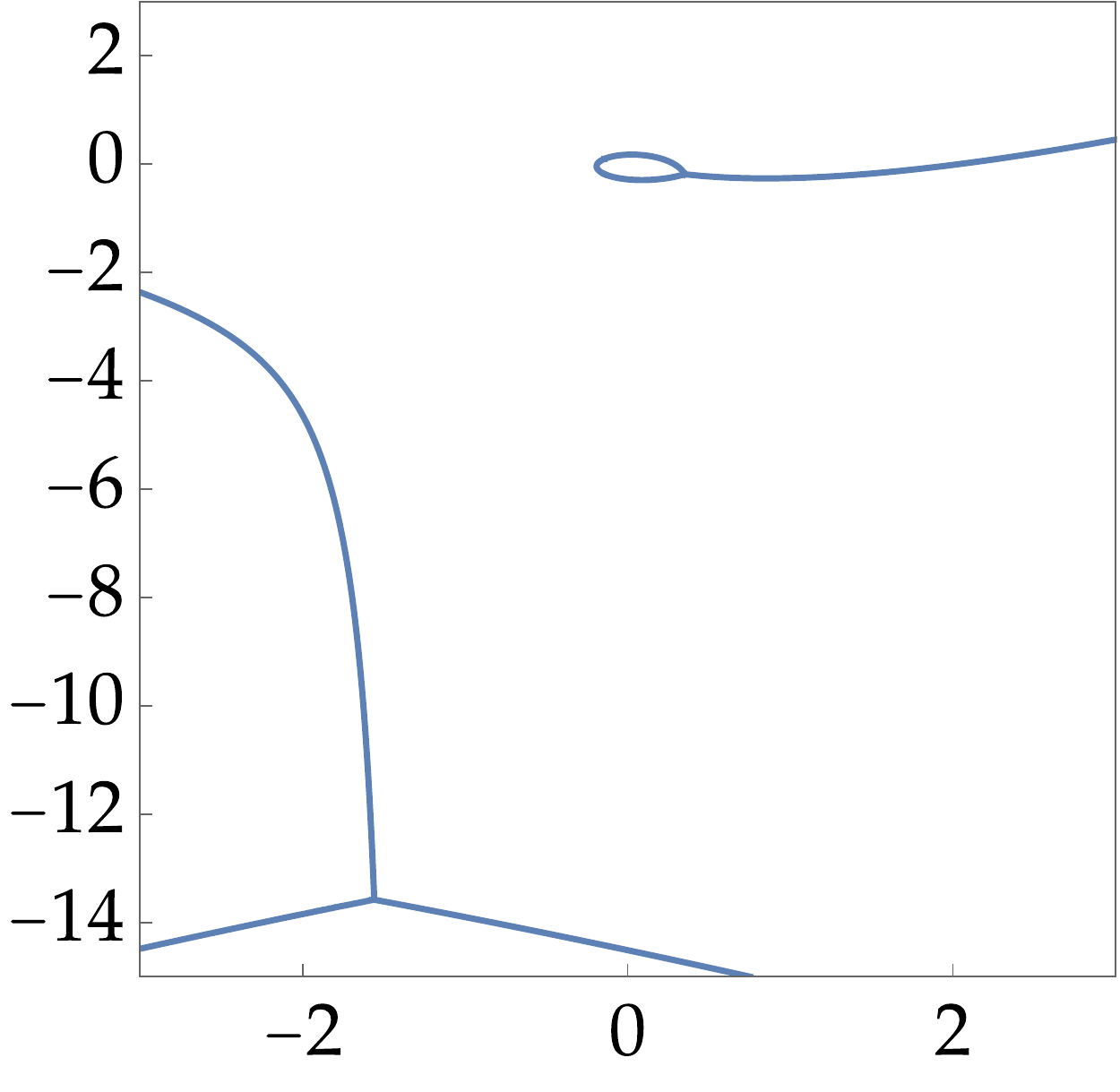}
\includegraphics[height=1.2in]{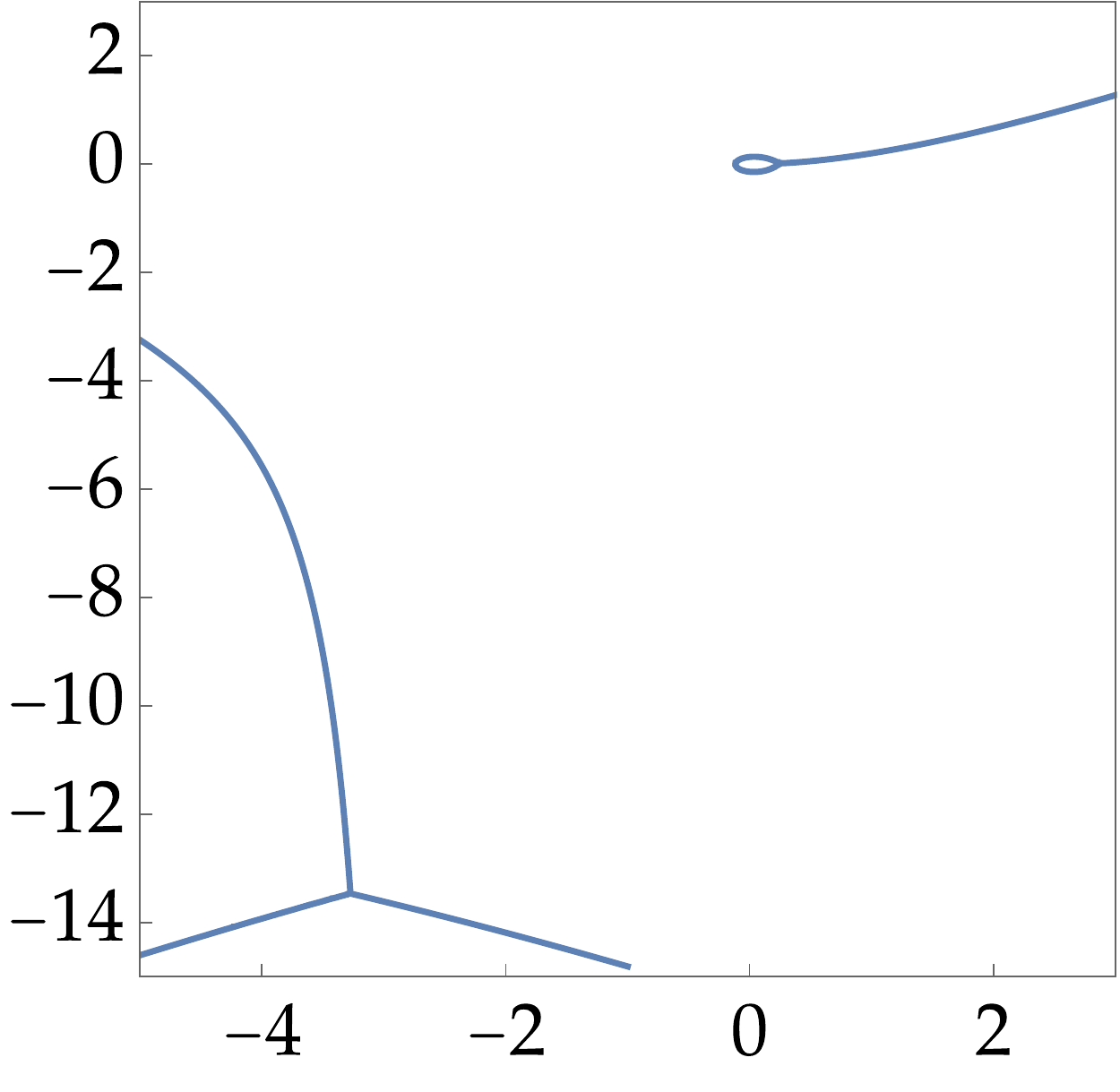}
\includegraphics[height=1.2in]{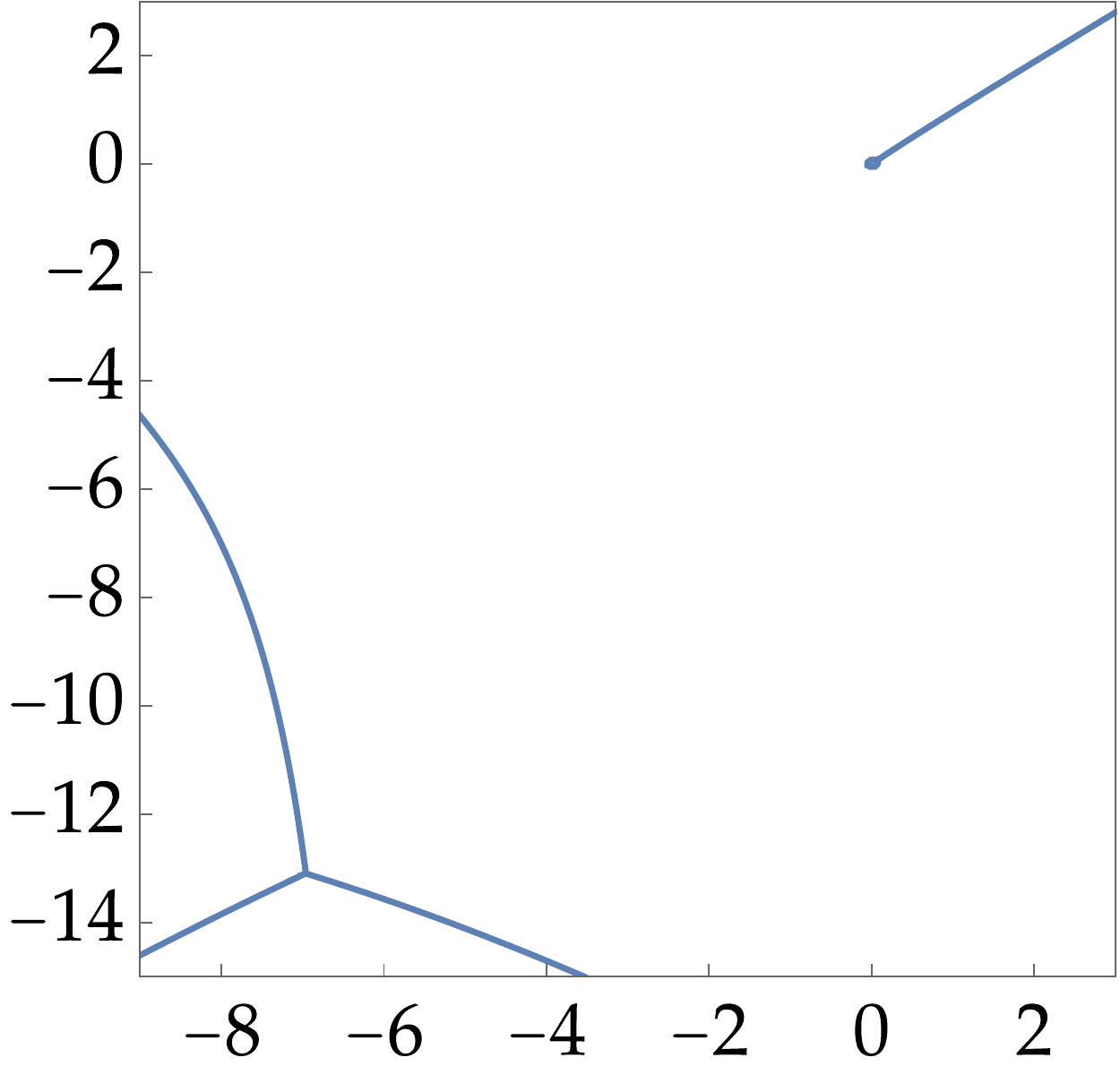}\\
\vspace{.18in}\\
\includegraphics[width=3in]{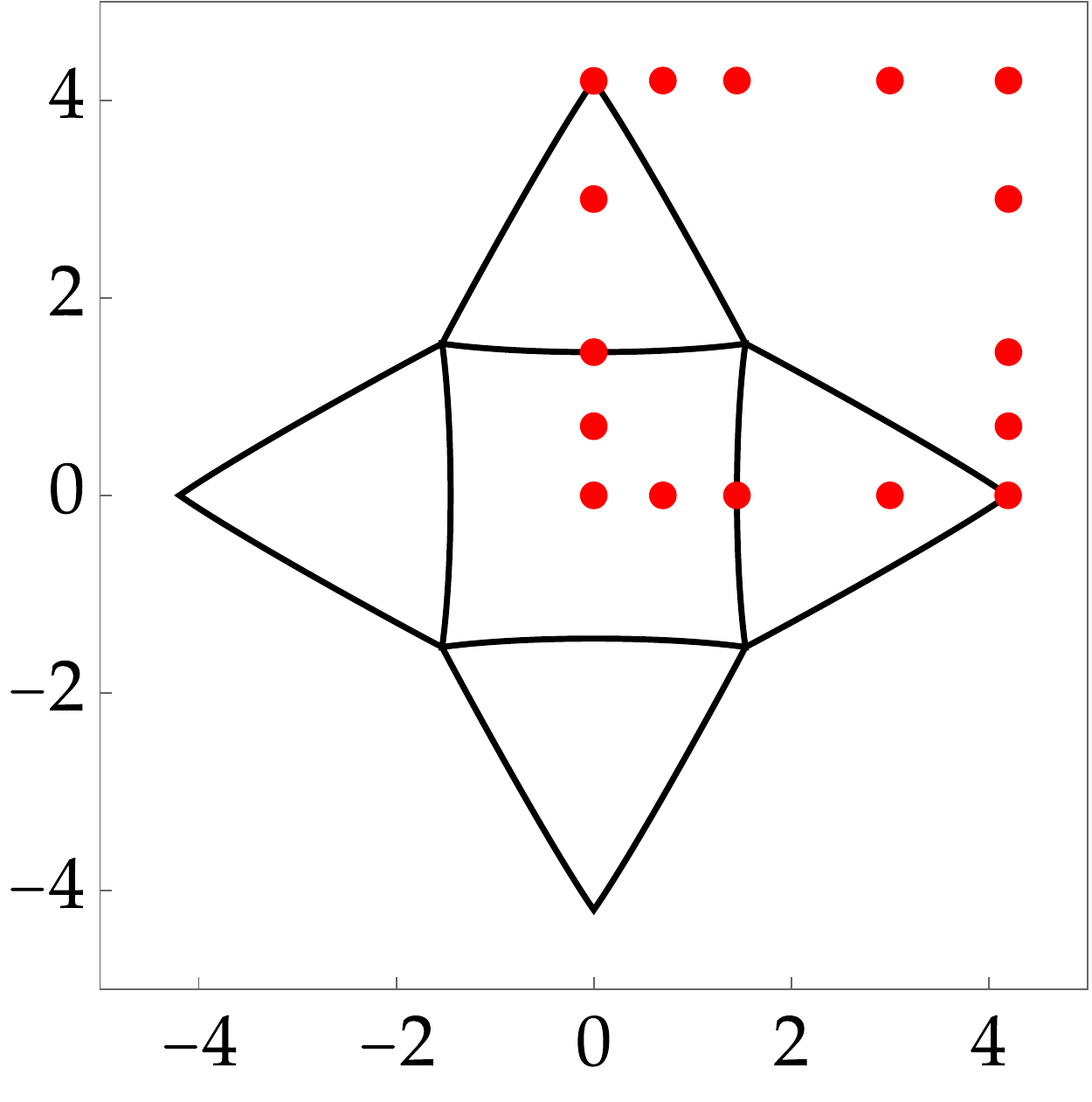}\\
\vspace{.18in}\\
\includegraphics[height=1.2in]{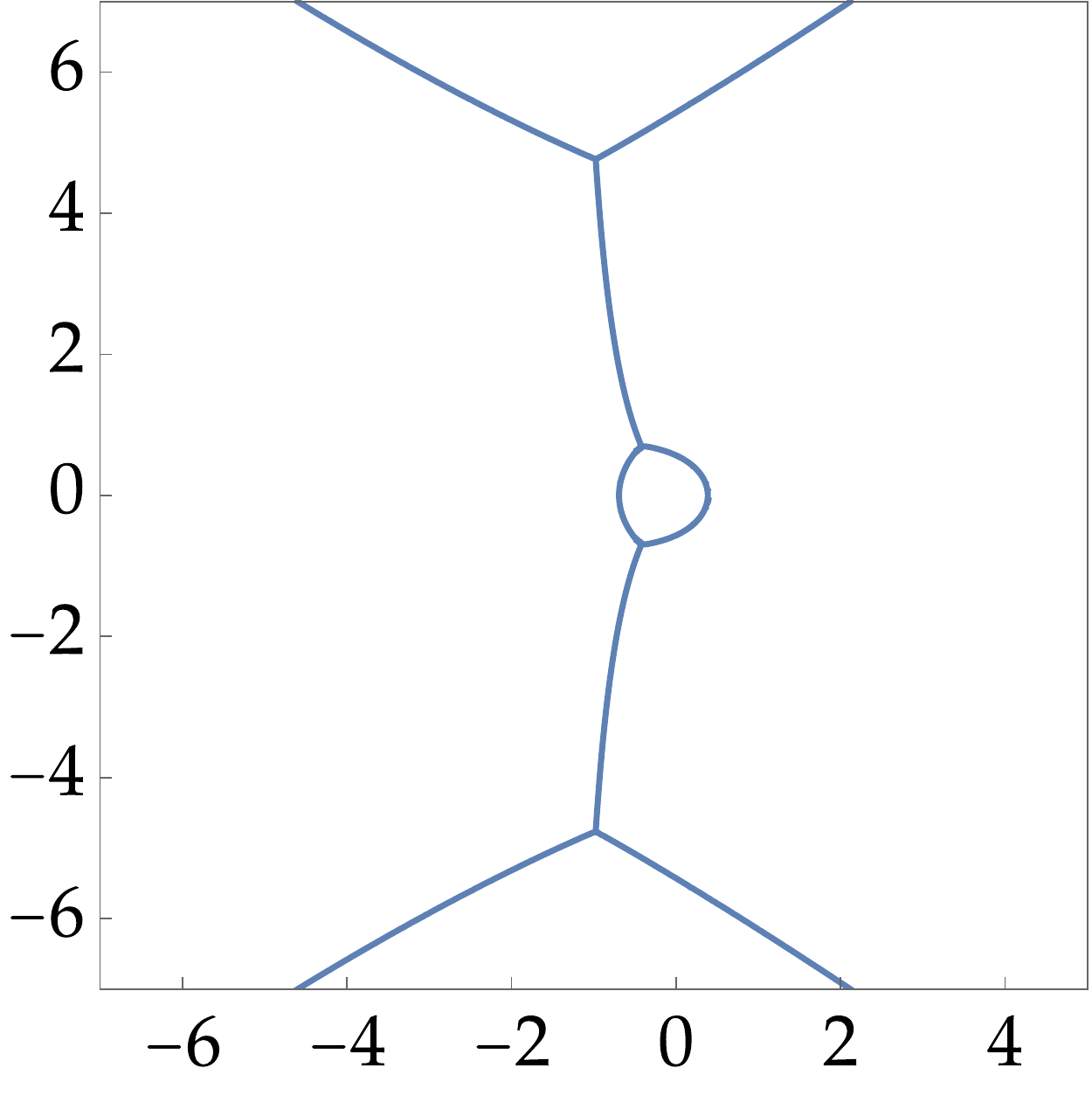}
\includegraphics[height=1.2in]{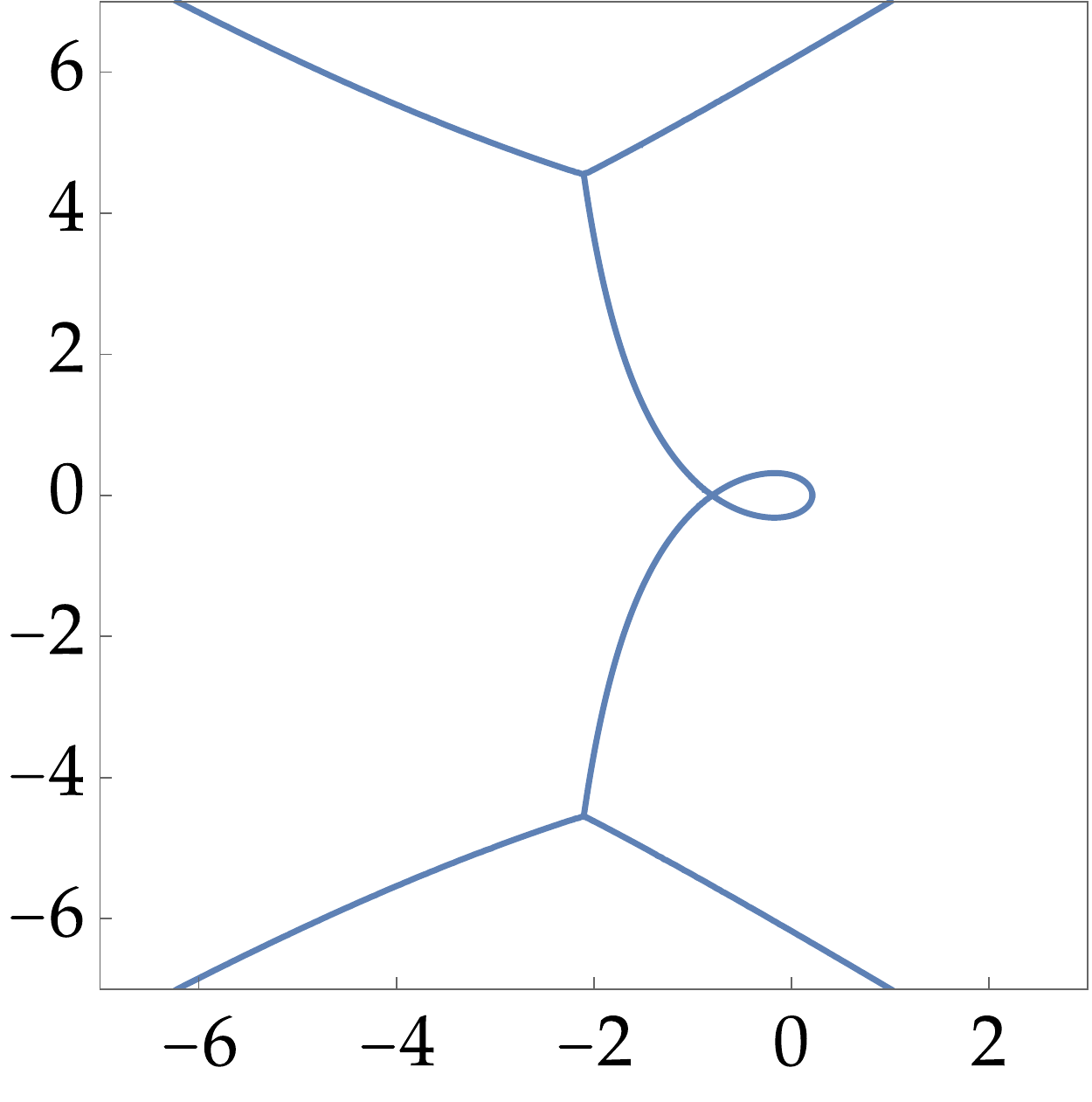}
\includegraphics[height=1.2in]{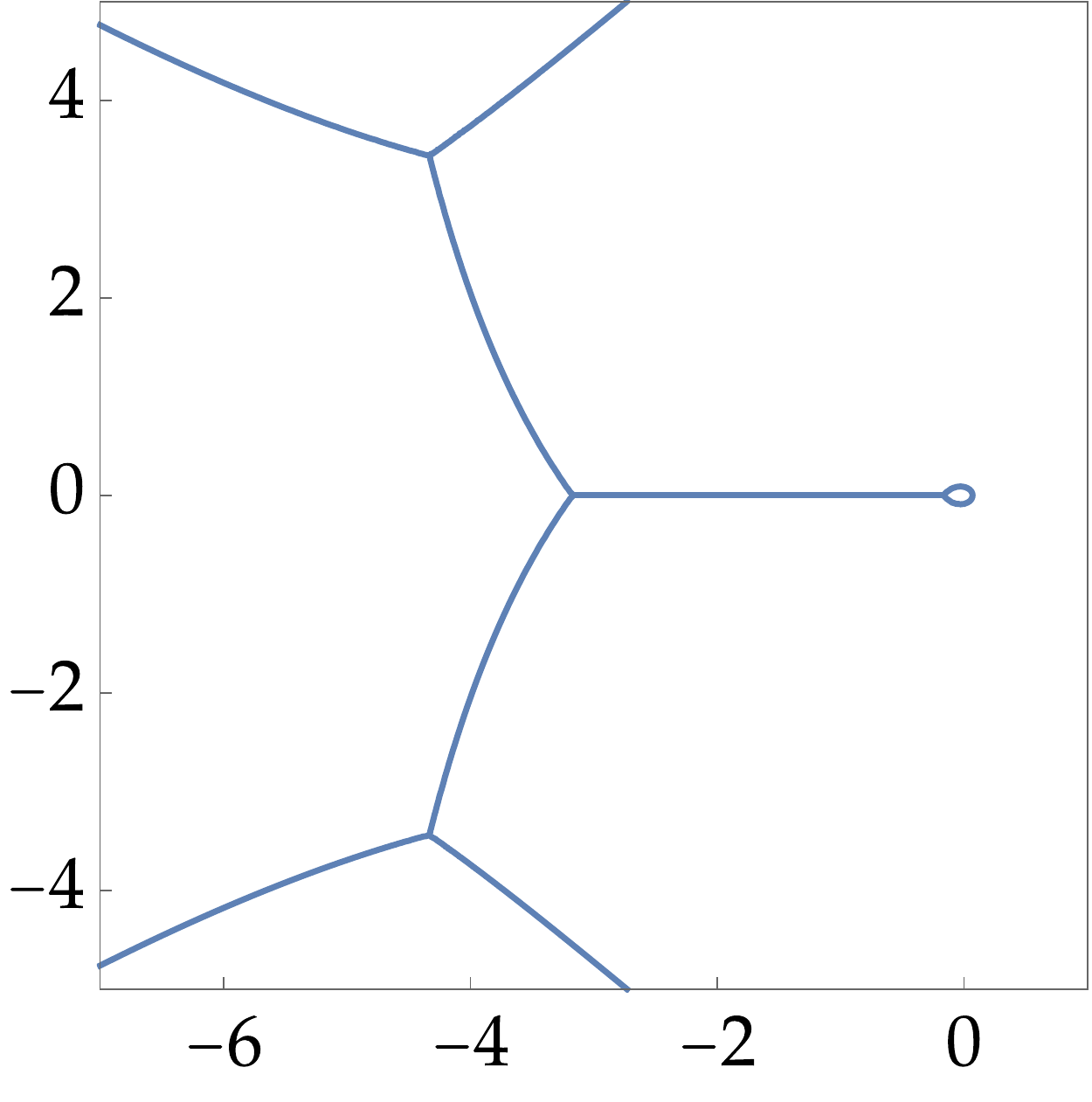}\\
\end{tabular}
\hspace{-.18in}
\begin{tabular}{c}
\includegraphics[height=1.2in]{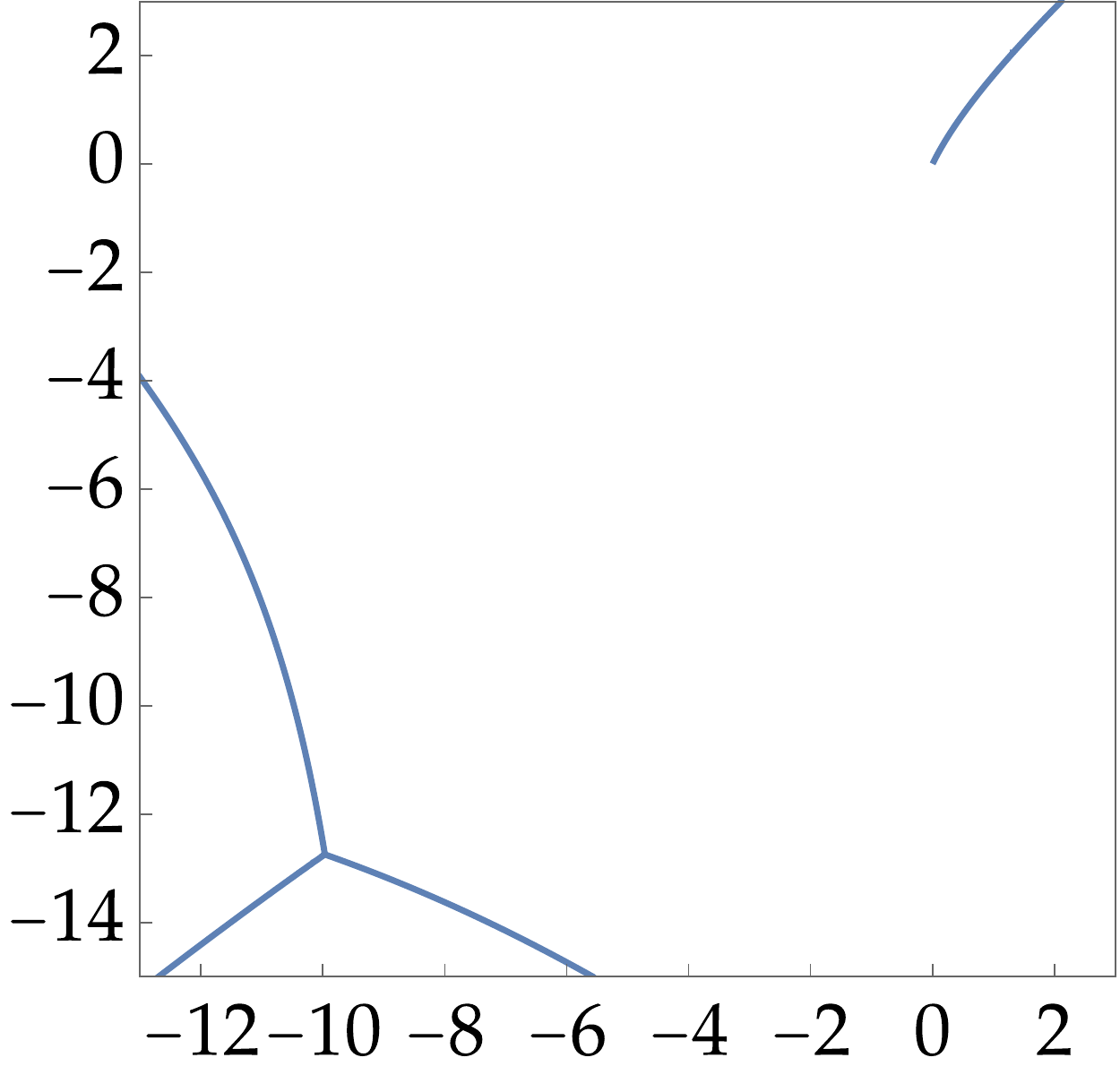}\\
\includegraphics[height=1.2in]{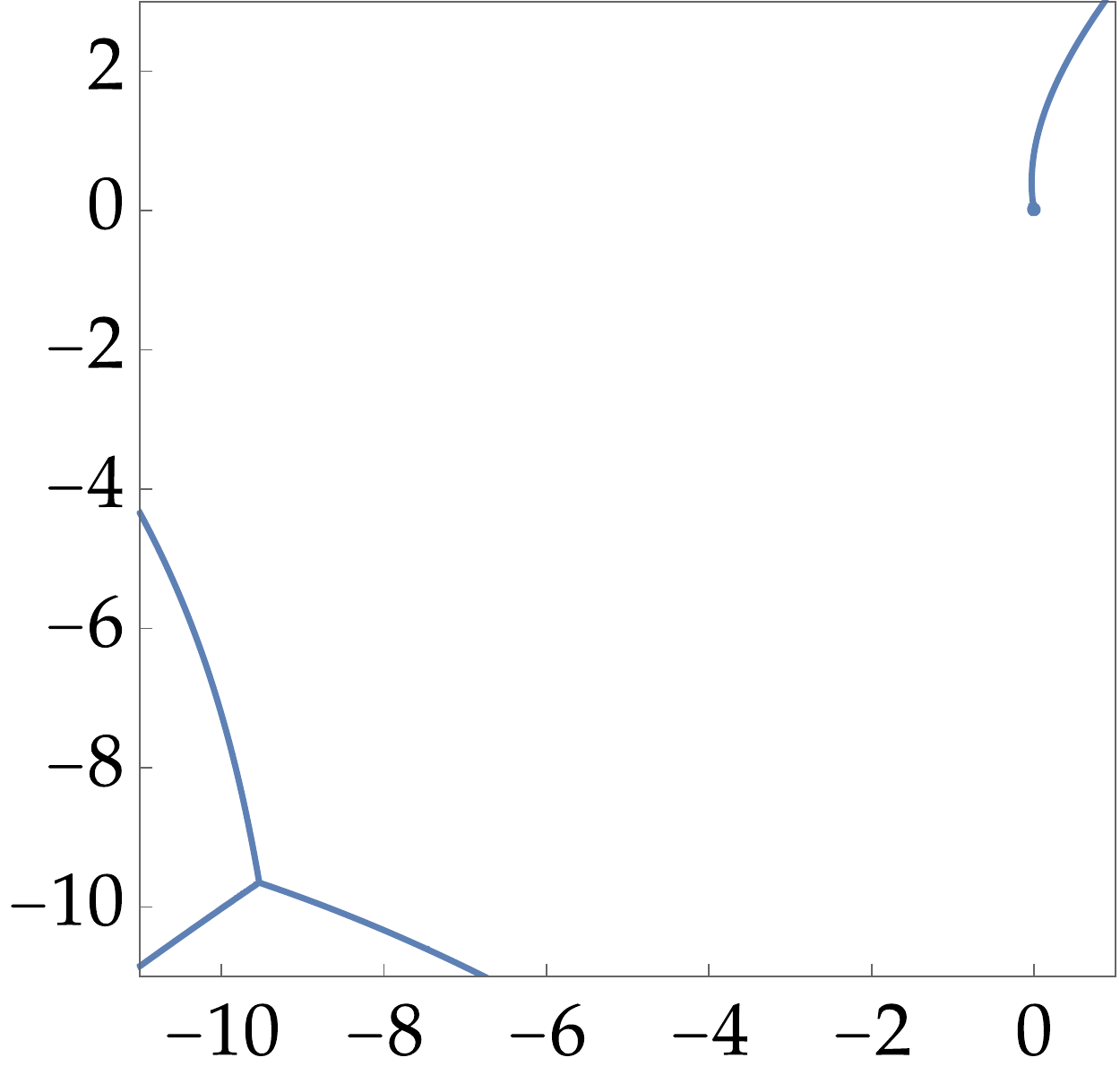}\\
\includegraphics[height=1.2in]{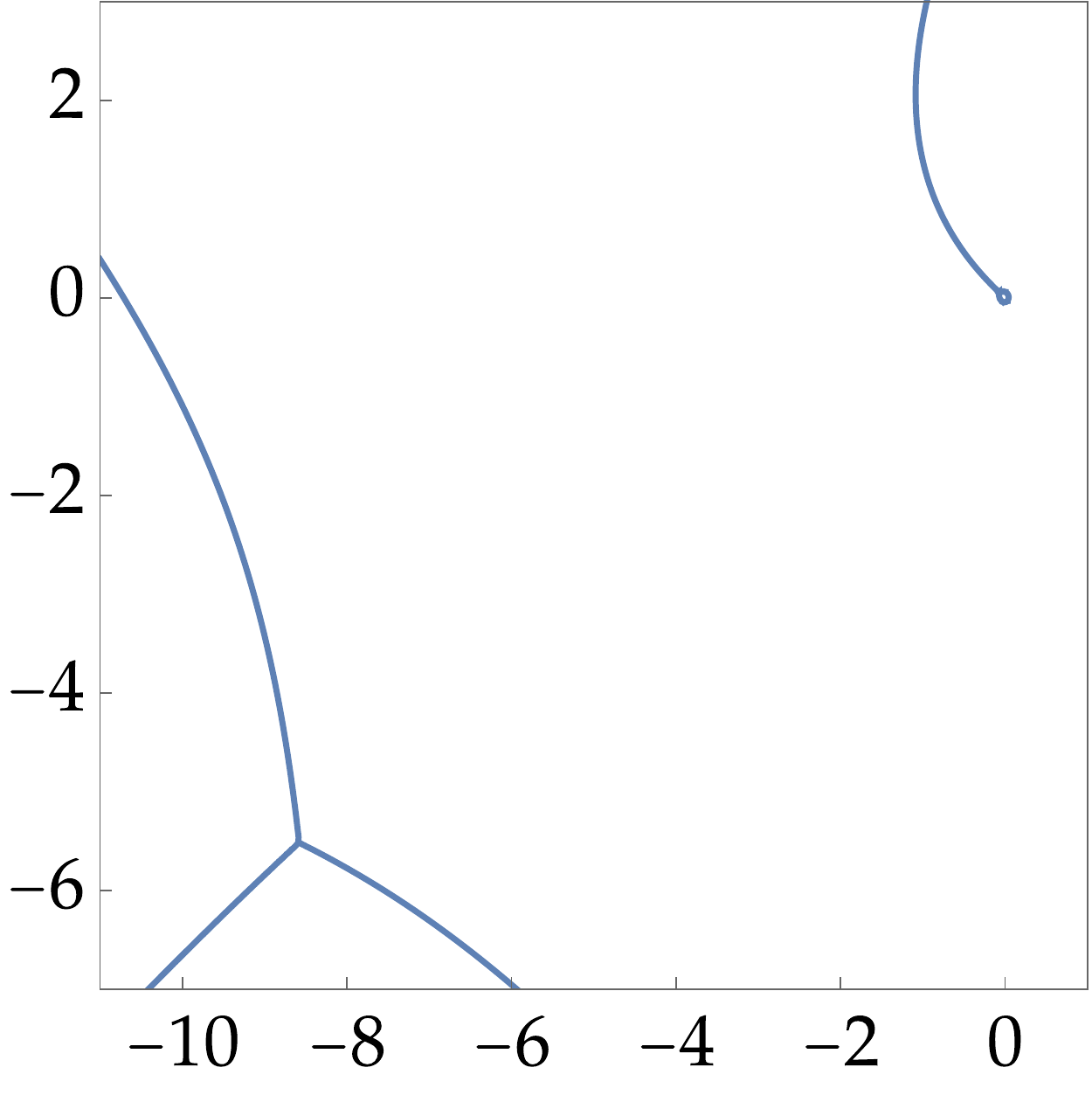}\\
\includegraphics[height=1.2in]{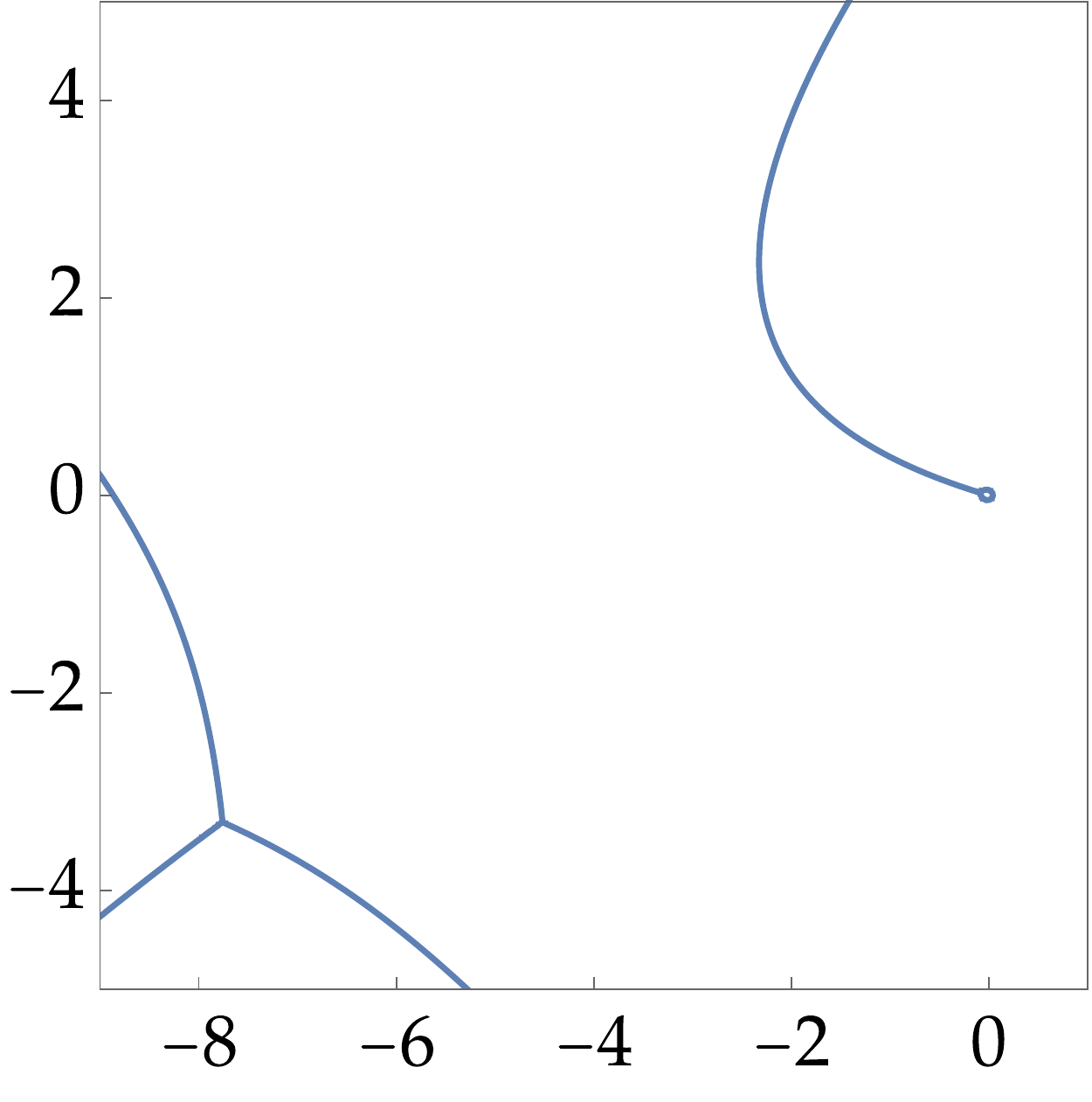}\\
\includegraphics[height=1.2in]{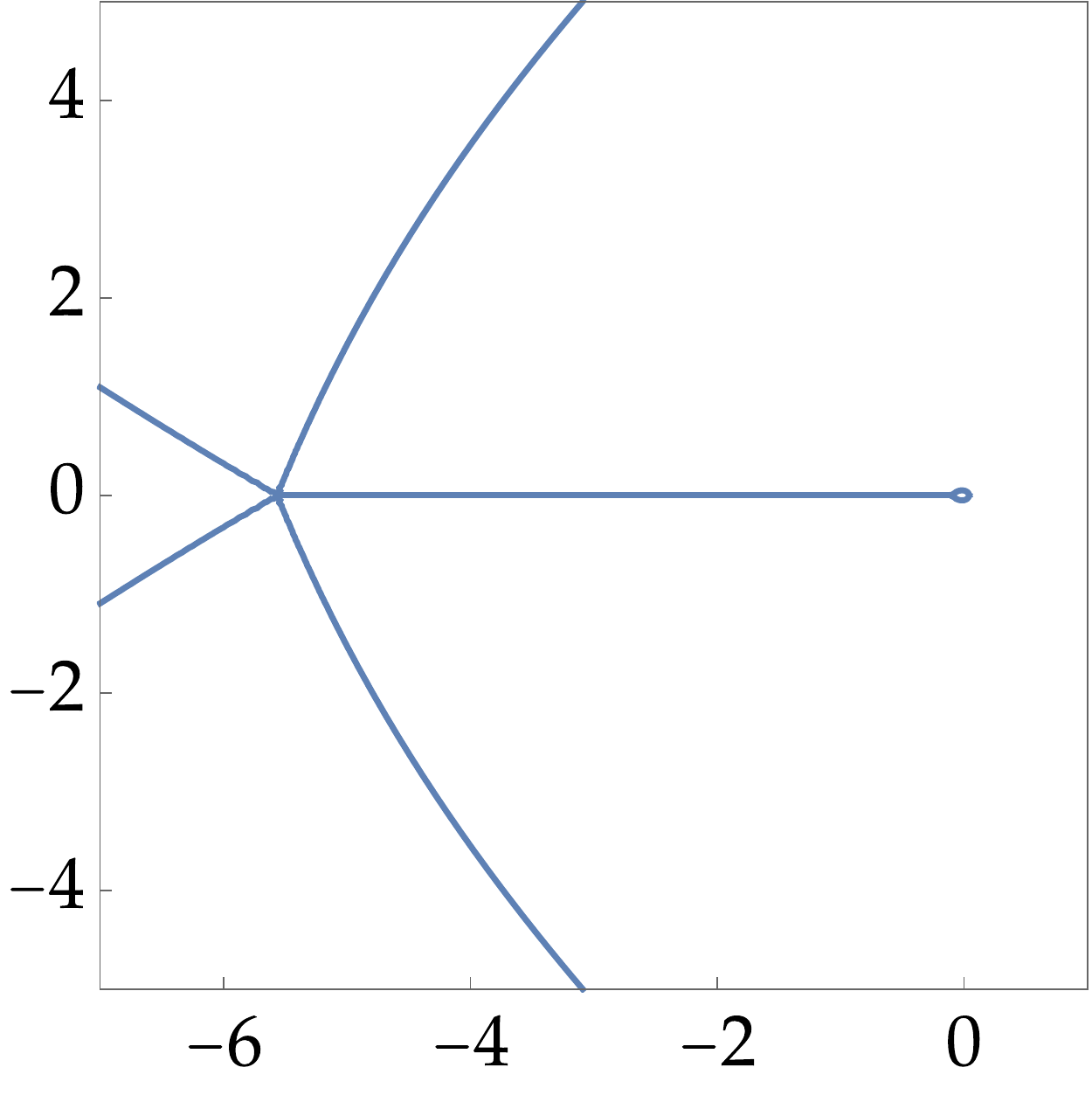}\\
\end{tabular}
\caption{Critical trajectories of $h'(z)^2\,\dd z^2$ emanating from simple roots of $P(z)$ for $\kappa=3$ for the gO family (for nongeneric values of $y_0$ on boundary curves we show all critical trajectories).    The same topological structure holds for $\kappa>1$.  Counterclockwise from top left:  $y\approx 4.19698i$, $y=3i$, $y\approx 1.45i$, $y=0.7i$, $y=0$, $y=0.7$, $y\approx 1.45$, $y=3$, $y\approx 4.19698$, $y=4.2+0.7i$, $y=4.2+1.45i$, $y=4.2+3i$, $y=4.2+4.2i$, $y=3+4.2i$, $y=1.45+4.2i$, $y=0.7+4.2i$.  Inset:  Boundaries of the regions $\rectangle(3)$, $\pm\TR(3)$, and $\pm\TI(3)$ in the $y$-plane.  The $y$-values corresponding to different trajectory plots are indicated by red dots.}
\label{fig:Trajectories-gO-k3}
\end{figure}

\begin{figure}[h]
\hspace{-0.2in}
\begin{tabular}{c}
\includegraphics[height=1.2in]{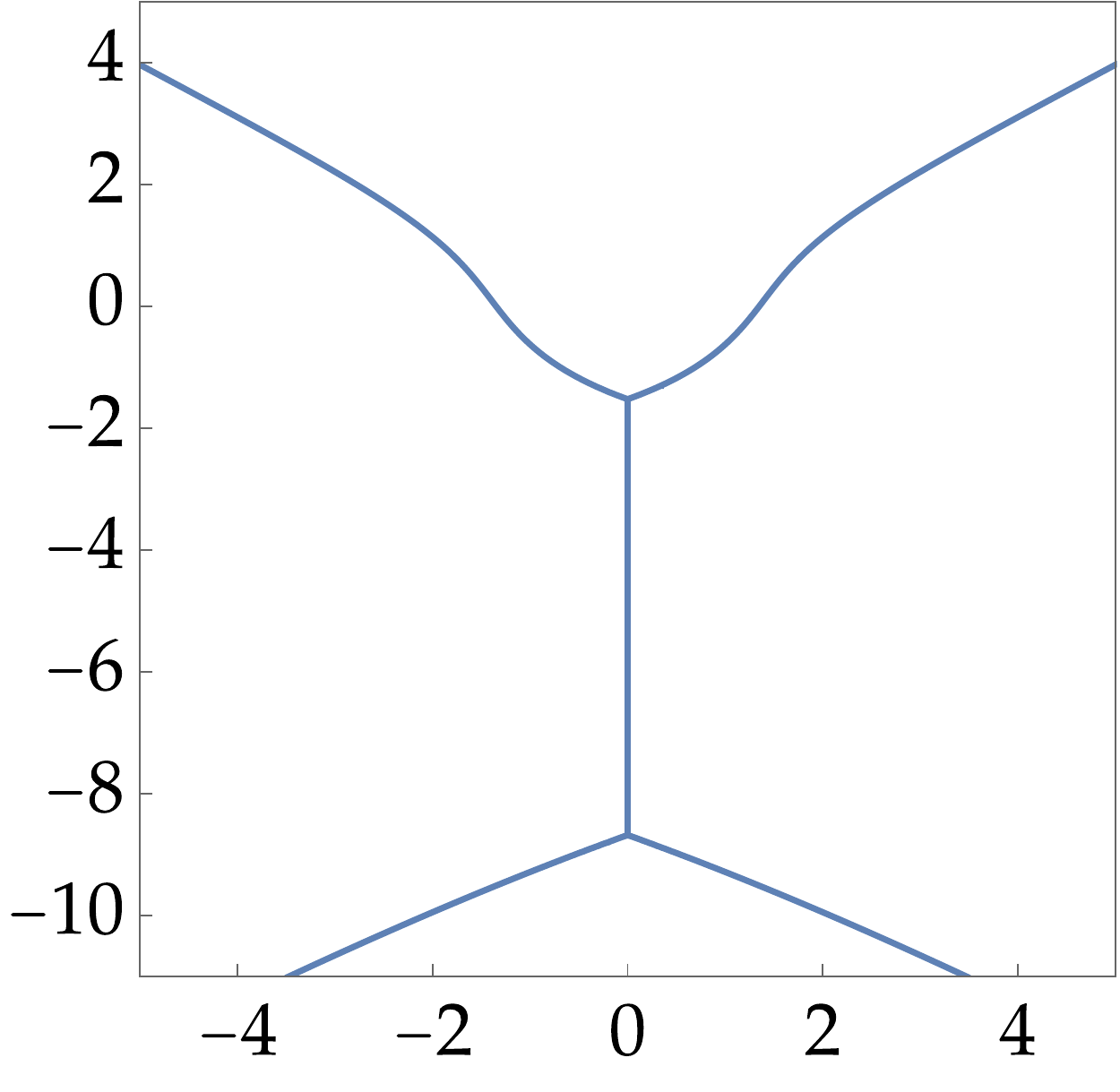}\\
\includegraphics[height=1.2in]{stokes-k3-y1p45i.pdf}\\
\includegraphics[height=1.2in]{stokes-k3-y0p7i.pdf}\\
\includegraphics[height=1.2in]{stokes-k3-y0.pdf}\\
\end{tabular}
\hspace{-.18in}
\begin{tabular}{c}
\includegraphics[height=1.2in]{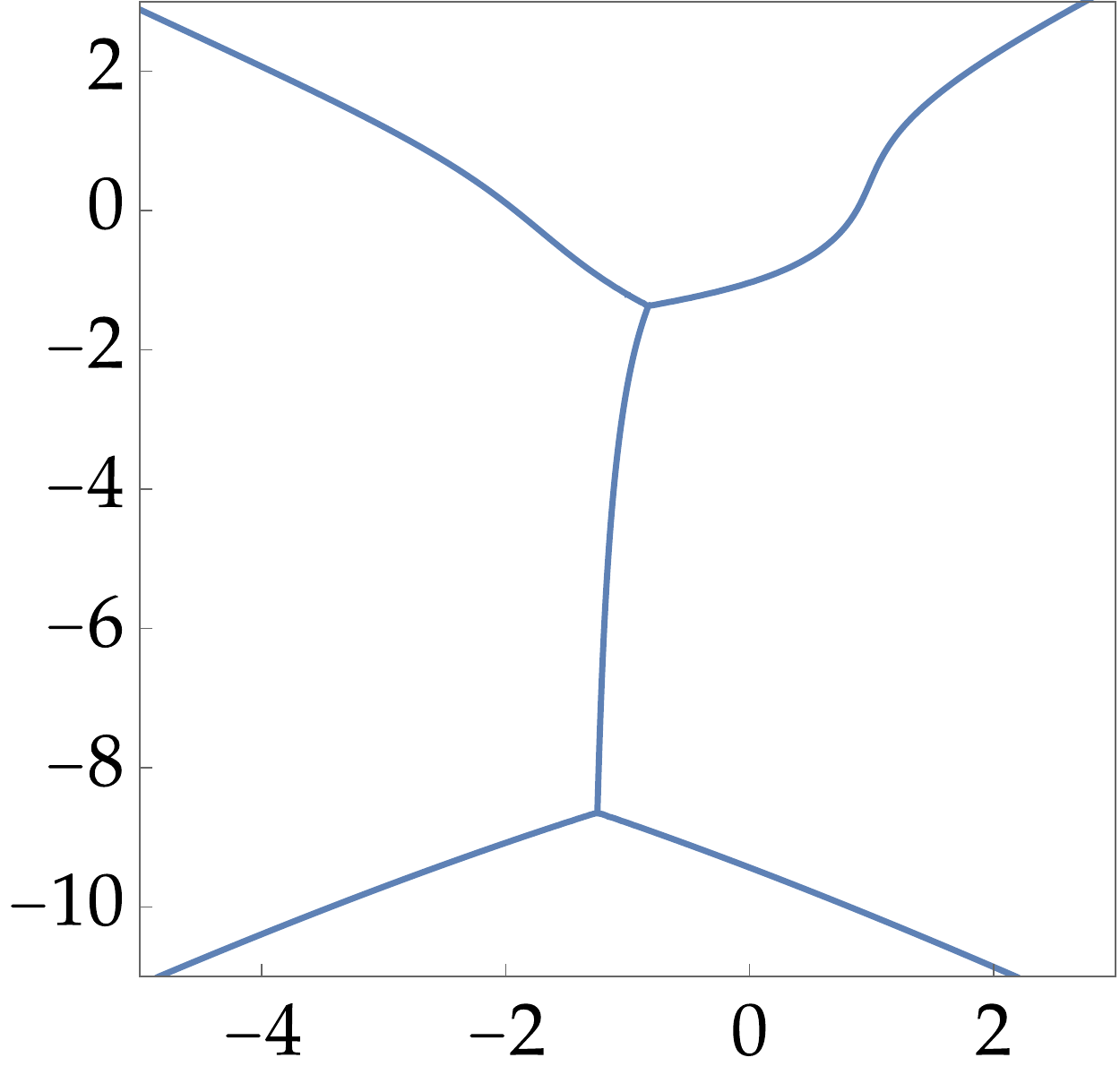}
\includegraphics[height=1.2in]{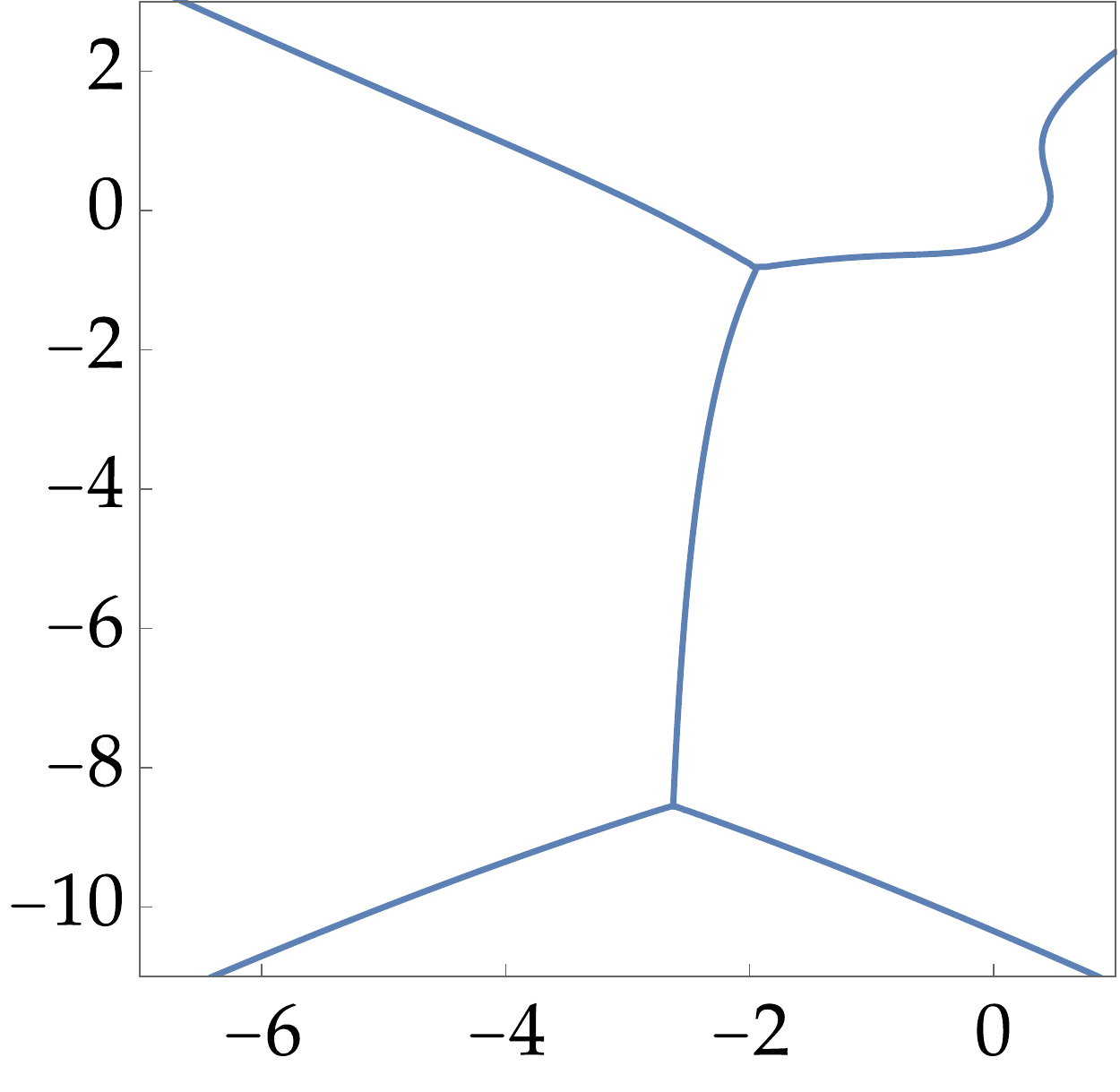}\\
\vspace{-.13in}\\
\includegraphics[width=2.3in]{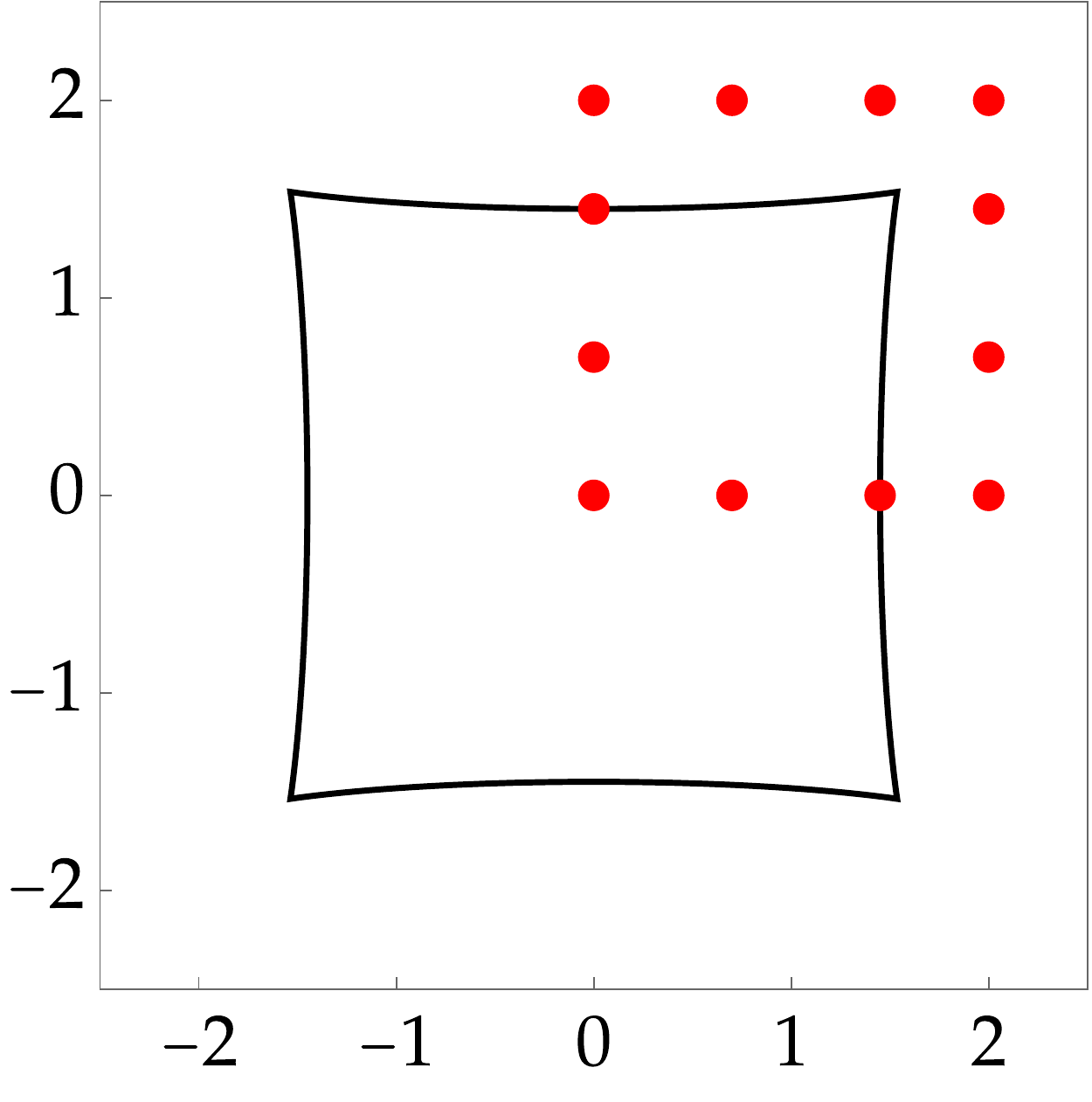}\\
\vspace{-.13in}\\
\includegraphics[height=1.2in]{stokes-k3-y0p7.pdf}
\includegraphics[height=1.2in]{stokes-k3-y1p45.pdf}
\end{tabular}
\hspace{-.18in}
\begin{tabular}{c}
\includegraphics[height=1.2in]{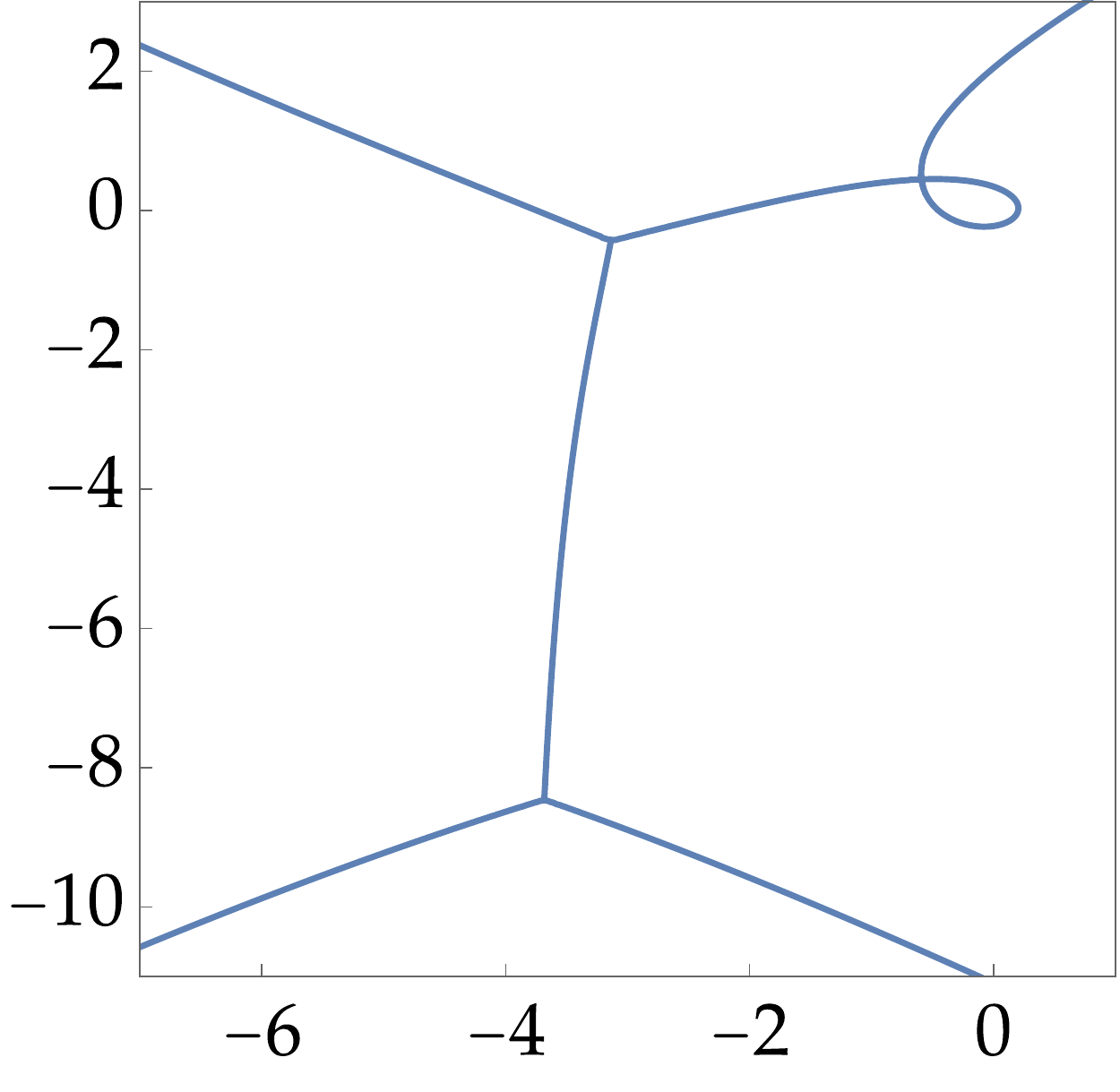}\\
\includegraphics[height=1.2in]{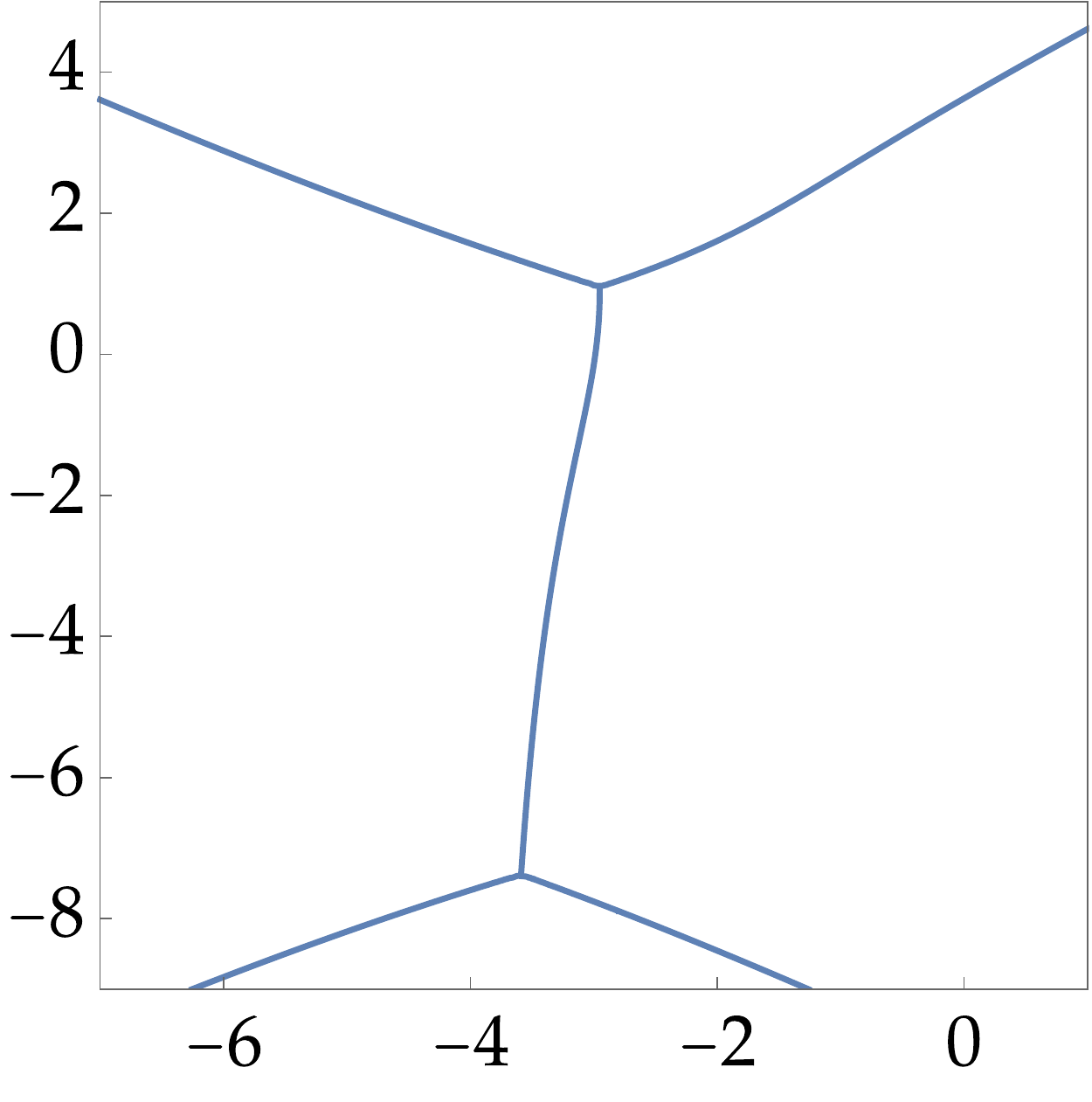}\\
\includegraphics[height=1.2in]{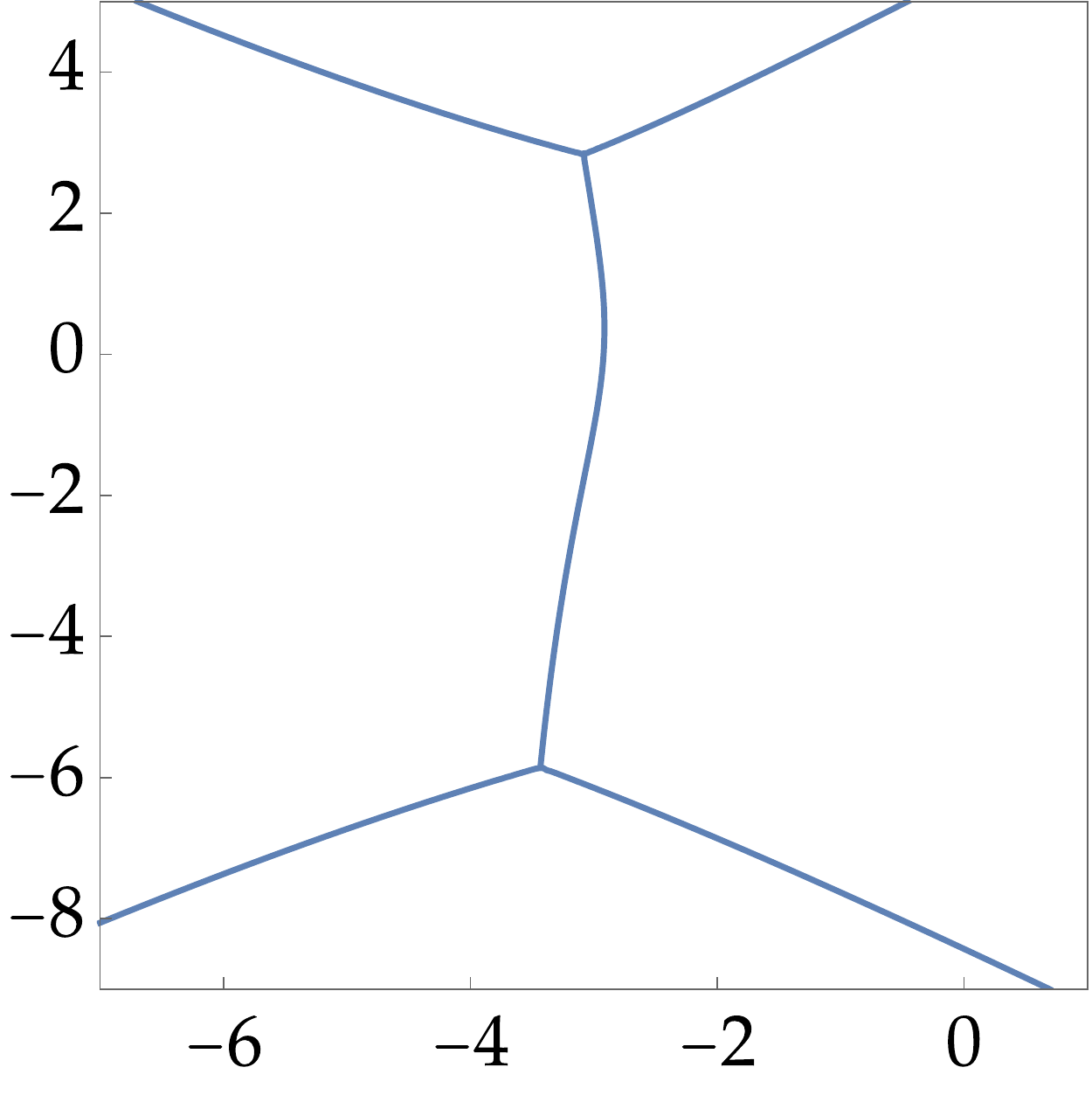}\\
\includegraphics[height=1.2in]{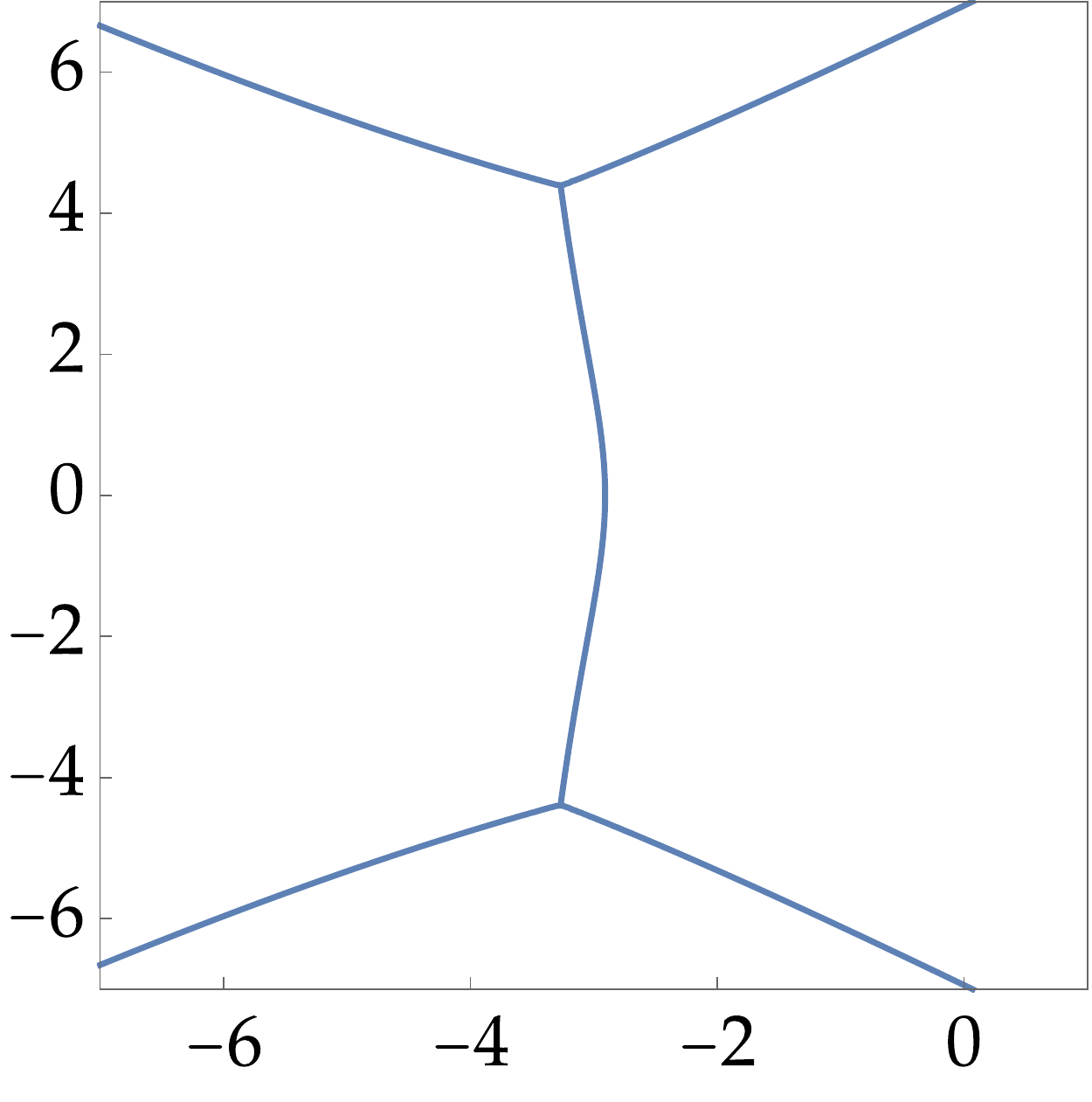}\\
\end{tabular}
\caption{Critical trajectories of $h'(z)^2\,\dd z^2$ emanating from (for generic $y_0$) simple roots of $P(z)$ for $\kappa=3$ for the gH family.  The same topological structure holds for $\kappa>1$.  Counterclockwise from top left:  $y=2i$, $y\approx 1.45i$, $y=0.7i$, $y=0$, $y=0.7$, $y\approx 1.45$, $y=2$, $y=2+0.7i$, $y=2+1.45i$, $y=2+2i$, $y=1.45+2i$, $y=0.7+2i$.  Inset:  Boundary of $\rectangle(3)$ in the $y$-plane.  The $y$-values corresponding to different trajectory plots are indicated by red dots.}
\label{Trajectories-gH-k3}
\end{figure}

%% file: IsomonodromyTheory.tex
\section{Isomonodromy theory for rational solutions of Painlev\'e-IV}
\label{app:Isomonodromy}
In Sections \ref{sec:Okamoto-RHP-basic-properties} and \ref{sec:Schlesinger} below, we develop several aspects of the isomonodromy method for the Painlev\'e-IV equation \eqref{p4}.  While much of this theory is in the literature (see, for instance \cite[Sections 5.1 and 6.3]{FokasIKN:2006}), to deal with the rational solutions we will add some important details by
\begin{itemize}
\item implementing the isomonodromy method for resonant Fuchsian singular points such as for the gH solutions because $\Theta_0\in\tfrac{1}{2}\mathbb{Z}$;
\item identifying non-differential formul\ae\ for the solution $u(x)$;
\item identifying a formula for the related solution $u_\tw(x)$;
\item determining conditions where various discrete isomonodromic transformations are well-defined.
\end{itemize}
Then, in Sections~\ref{sec:RHPgO} and \ref{sec:RHPgH} we apply the method as described in the beginning of Section~\ref{sec:rational-isomonodromy}, leveraging the $x$-equation in the Lax pair to compute monodromy data and then applying Schlesinger transformations to prove Theorems~\ref{thm:gO-RHP} and \ref{thm:gH-RHP}.

\subsection{Basic properties of Riemann-Hilbert Problem~\ref{rhp:general}}
\label{sec:Okamoto-RHP-basic-properties}
We start with the following statement asserting that uniqueness of the solution of Riemann-Hilbert Problem~\ref{rhp:general} is automatic, and existence is an ``almost all or nothing'' phenomenon with respect to $x\in\mathbb{C}$.
\begin{proposition}
%Fix $(m,n)\in\mathbb{Z}^2$.
The following statements are true:
\begin{itemize}
\item Given $x\in\mathbb{C}$, every solution of Riemann-Hilbert Problem~\ref{rhp:general} satisfies $\det(\mathbf{Y}(\lambda;x))=1$ for all $\lambda\in\mathbb{C}\setminus\Sigma$.
\item There exists at most one solution of Riemann-Hilbert Problem~\ref{rhp:general} for each $x\in\mathbb{C}$.
% and any solution satisfies $\mathbf{Y}^{(m,n)}(\lambda^*;x^*)=\mathbf{Y}^{(m,n)}(\lambda;x)^*$ (elementwise complex conjugation). 
\item Either 
\begin{itemize}
\item Riemann-Hilbert Problem~\ref{rhp:general} has no solution for any $x\in\mathbb{C}$, or 
\item Riemann-Hilbert Problem~\ref{rhp:general} has a solution for $x\in\mathbb{C}\setminus\mathcal{D}$, where $\mathcal{D}\subset\mathbb{C}$ is a discrete set, for each $\lambda\in\mathbb{C}\setminus\Sigma$ the solution is an analytic function of $x\in\mathbb{C}\setminus\mathcal{D}$, and the solution admits an asymptotic expansion to all orders as $\lambda\to\infty$ in the sense that
\eq
\mathbf{Y}(\lambda;x)\lambda^{\Theta_\infty\sigma_3}\sim\mathbb{I}+\sum_{k=1}^\infty\mathbf{Y}^\infty_k(x)\lambda^{-k},\quad\lambda\to\infty,
\label{eq:Y-infty-expand}
\endeq
and a convergent expansion in a neighborhood of $\lambda=0$ of the form
\eq
\mathbf{Y}(\lambda;x)\lambda^{-\Theta_0\sigma_3}=\sum_{k=0}^\infty\mathbf{Y}^0_k(x)\lambda^k 
\label{eq:Y-zero-expand}
\endeq
in which all matrix coefficients $\mathbf{Y}^\infty_k(x)$ for $k\in\mathbb{Z}_{>0}$ and $\mathbf{Y}^0_k(x)$ for $k\in\mathbb{Z}_{\ge 0}$ are meromorphic functions of $x$ with poles in $\mathcal{D}$.  Both expansions \eqref{eq:Y-infty-expand} and \eqref{eq:Y-zero-expand} are differentiable term-by-term with respect to both $\lambda$ and $x$.
\end{itemize}
\end{itemize}
\label{prop:RHP-basic}
\end{proposition} 
\begin{proof}
The first and second statements are both simple consequences of Liouville's Theorem.  In light of the identities \eqref{eq:consistency}, the third statement is a consequence of the Analytic Fredholm Theorem applied to a system of singular integral equations equivalent to Riemann-Hilbert Problem~\ref{rhp:general}.  See \cite{FokasIKN:2006} for details.
\end{proof} 
From the uniqueness one easily obtains the following.
\begin{corollary}
Suppose that the arcwise-constant matrix $\mathbf{V}(\lambda)$ satisfies $\mathbf{V}(\lambda^*)^*=\mathbf{V}(\lambda)^{-1}$, where $\mathbf{V}^*$ denotes elementwise complex conjugation.  If Riemann-Hilbert Problem~\ref{rhp:general} has a solution $\mathbf{Y}(\lambda;x)$ for some $x\in\mathbb{C}$, then it does also for $x^*$, and $\mathbf{Y}(\lambda^*;x^*)=\mathbf{Y}(\lambda;x)^*$. 
\label{cor:Schwarz}
\end{corollary}

The next result shows that simultaneous differential equations of the form \eqref{eq:PIV-Lax-Pair} can be deduced directly from the conditions of Riemann-Hilbert Problem~\ref{rhp:general}.
\begin{proposition}
Suppose that Riemann-Hilbert Problem~\ref{rhp:general} has a solution for $x\in\mathbb{C}\setminus\mathcal{D}$, where $\mathcal{D}$ is a discrete set.  Then $\mathbf{\Psi}(\lambda,x):=\mathbf{Y}(\lambda;x)\ee^{(\frac{1}{2}\lambda^2+x\lambda)\sigma_3}$ satisfies the Lax equations
\eq
\frac{\partial\mathbf{\Psi}}{\partial x}(\lambda,x)=\left(\lambda\sigma_3 + [\mathbf{Y}^\infty_1(x),\sigma_3]\right)\mathbf{\Psi}(\lambda,x)
\label{eq:Lax-x}
\endeq
and
\eq
\frac{\partial\mathbf{\Psi}}{\partial\lambda}(\lambda,x)=
\left(\lambda\sigma_3 + x\sigma_3+[\mathbf{Y}^\infty_1(x),\sigma_3]+\lambda^{-1}\mathbf{\Lambda}_1(x)\right)\mathbf{\Psi}(\lambda,x)
\label{eq:Lax-lambda}
\endeq
for each $(\lambda,x)\in(\mathbb{C}\setminus\Sigma)\times (\mathbb{C}\setminus\mathcal{D})$, where the coefficient matrix $\mathbf{\Lambda}_1(x)$ has two equivalent representations:
\eq
\begin{split}
\mathbf{\Lambda}_1(x)&=x[\mathbf{Y}^\infty_1(x),\sigma_3]+[\mathbf{Y}^\infty_2(x),\sigma_3]-[\mathbf{Y}^\infty_1(x),\sigma_3]\mathbf{Y}^\infty_1(x)-\Theta_\infty\sigma_3\\
&=\Theta_0\mathbf{Y}^0_0(x)\sigma_3\mathbf{Y}^0_0(x)^{-1}.
\end{split}
\label{eq:M-define}
\endeq
\label{prop:LaxEquations}
\end{proposition}
\begin{proof}
The function $\mathbf{\Psi}(\lambda,x)$ is analytic in both variables $(\lambda,x)$ exactly where $\mathbf{Y}(\lambda;x)$ is, and it satisfies jump conditions analogous to those satisfied by $\mathbf{Y}(\lambda;x)$ but with the exponential factors omitted.  These jump conditions are independent of both $\lambda$ and $x$, so the partial derivatives on the left-hand sides of \eqref{eq:Lax-x} and \eqref{eq:Lax-lambda} satisfy the same jump conditions as does $\mathbf{\Psi}(\lambda,x)$ itself.  Taking into account the classical nature of the boundary values on the jump contour $\Sigma$ and the unit determinant of $\mathbf{\Psi}(\lambda,x)$ guaranteed by the first property of Proposition~\ref{prop:RHP-basic}, it then follows easily that $\mathbf{X}(\lambda,x):=\partial_x\mathbf{\Psi}(\lambda,x)\cdot\mathbf{\Psi}(\lambda,x)^{-1}$ and $\mathbf{\Lambda}(\lambda,x):=\partial_\lambda\mathbf{\Psi}(\lambda,x)\cdot\mathbf{\Psi}(\lambda,x)^{-1}$ are both analytic functions of $\lambda\in\mathbb{C}\setminus\{0\}$ for every $x\in \mathcal{D}$.  Using \eqref{eq:Y-infty-expand} and its derivatives with respect to $x$ and $\lambda$ shows that 
\eq
\mathbf{X}(\lambda,x)=\lambda\sigma_3 + [\mathbf{Y}^\infty_1(x),\sigma_3] + \bo(\lambda^{-1}),\quad\lambda\to\infty
\label{eq:X-infty-expand}
\endeq
and
\eq
\mathbf{\Lambda}(\lambda,x)=\lambda\sigma_3 + x\sigma_3 + [\mathbf{Y}^\infty_1(x),\sigma_3] +
\lambda^{-1}\mathbf{\Lambda}_1(x) + \bo(\lambda^{-2}),\quad\lambda\to\infty
\label{eq:Lambda-infty-expand}
\endeq
where $\mathbf{\Lambda}_1(x)$ is given by the first line of \eqref{eq:M-define}.  Likewise, using \eqref{eq:Y-zero-expand} and its derivatives with respect to $x$ and $\lambda$ shows that
\eq
\mathbf{X}(\lambda,x)=\bo(1),\quad\lambda\to 0
\endeq
and
\eq
\mathbf{\Lambda}(\lambda,x)=\lambda^{-1}\Theta_0\mathbf{Y}^0_0(x)\sigma_3\mathbf{Y}^0_0(x)^{-1} + \bo(1),\quad\lambda\to 0.
\label{eq:Lambda-zero-expand}
\endeq
Note that $\det(\mathbf{Y}^0_0(x))=1$ according to the first statement of Proposition~\ref{prop:RHP-basic}.
Hence by Liouville's Theorem, $\mathbf{X}(\lambda,x)$ is a linear function of $\lambda$ given by the two explicit terms on the right-hand side of \eqref{eq:X-infty-expand}, while $\mathbf{\Lambda}(\lambda,x)$ is a Laurent polynomial of degree $(1,1)$ given by the four explicit terms on the right-hand side of \eqref{eq:Lambda-infty-expand}.  Comparing the latter with \eqref{eq:Lambda-zero-expand} then gives the second equality in \eqref{eq:M-define} and completes the proof.
\end{proof}

\begin{proposition}
Suppose that Riemann-Hilbert Problem~\ref{rhp:general} has a solution for $x\in\mathbb{C}\setminus\mathcal{D}$, where $\mathcal{D}$ is a discrete set.  Set
\eq
\begin{split}
y(x)&:=-2Y^\infty_{1,12}(x)\\ z(x)&:=\Theta_0+\Theta_\infty-2Y^\infty_{1,12}(x)Y^\infty_{1,21}(x)=-2\Theta_0Y^0_{0,12}(x)Y^0_{0,21}(x)\\
u(x)&:=2Y^\infty_{1,22}(x)-2x-2\frac{Y^\infty_{2,12}(x)}{Y^\infty_{1,12}(x)}=-2\Theta_0\frac{Y^0_{0,11}(x)Y^0_{0,12}(x)}{Y^\infty_{1,12}(x)}.
\end{split}
\label{eq:yzu-define}
\endeq
The definition of $u(x)$ is determinate if and only if $Y^\infty_{1,12}(x)$ does not vanish identically.
In this case, $y(x)$, $z(x)$, and $u(x)$ are meromorphic on $\mathbb{C}\setminus\mathcal{D}$; more precisely $y(x)$ is analytic and not identically vanishing, $z(x)$ is analytic, and $u(x)$ has a discrete set of poles in $\mathbb{C}\setminus\mathcal{D}$ corresponding to the necessarily isolated zeros of the analytic function $Y^\infty_{1,12}(x)$ proportional to $y(x)$.
Assuming furthermore that $u(x)$ does not vanish identically, the Lax equations \eqref{eq:Lax-x} and \eqref{eq:Lax-lambda} take exactly the form \eqref{eq:PIV-Lax-Pair} subject to \eqref{eq:JM-A}--\eqref{eq:JM-U}.  Furthermore, the differential identities
\eq
\begin{split}
y'&=-(2x+u)y\\
z'&=(\Theta_0+\Theta_\infty)u+4\Theta_0\frac{z}{u}-uz-2\frac{z^2}{u}\\
u'&=4\Theta_0+2xu+u^2-4z
\end{split}
\label{eq:PIV-system}
\endeq
hold for every $x\in\mathbb{C}\setminus\mathcal{D}$ that is not a pole or zero of $u(x)$, and $u(x)$ is a meromorphic solution of the Painlev\'e-IV equation in the form \eqref{p4}.
\label{prop:DiffEqs}
\end{proposition}
\begin{proof}
The definitions \eqref{eq:yzu-define} amount to nothing more than a parametrization of the coefficients in the Lax equations \eqref{eq:Lax-x} and \eqref{eq:Lax-lambda} from Proposition~\ref{prop:LaxEquations}, subject to the condition following from the second line of \eqref{eq:M-define} that $\mathbf{\Lambda}_1(x)$ has trace zero and determinant $-\Theta_0^2$; the alternate forms of $z(x)$ and $u(x)$ then follow from comparing the two different representations of $\Lambda_{1,11}(x)$ and $\Lambda_{1,12}(x)$ given on the two lines of \eqref{eq:M-define}.  The relation between the basic analyticity properties of these functions and the statement that $Y^\infty_{1,12}(x)$ is an analytic function on $\mathbb{C}\setminus\mathcal{D}$ that does not vanish identically then follow from Proposition~\ref{prop:RHP-basic}.  The three equations \eqref{eq:PIV-system} are exactly those arising from the zero-curvature compatibility condition
$\mathbf{\Lambda}_x-\mathbf{X}_\lambda + [\mathbf{\Lambda},\mathbf{X}]=\mathbf{0}$ upon separating out the coefficients of the different powers of $\lambda$ that appear, and these make sense provided neither $y(x)$ nor $u(x)$ vanishes identically.   Finally, elimination of $z$ in favor of $u$ by using the third equation of \eqref{eq:PIV-system} in the second yields the Painlev\'e-IV equation \eqref{p4} on $u(x)$.
\end{proof}

\begin{remark}
\label{rem:alternate}
%Using the $(1,2)$-entry of \eqref{eq:M-define}, we can equivalently express $u(x)$ in the form
%\begin{equation}
%u(x)=-2\Theta_0\frac{Z^{(m,n)}_{0,11}(x)Z^{(m,n)}_{0,12}(x)}{Y^{(m,n)}_{1,12}(x)}.
%\label{eq:u-alt}
%\end{equation}
The formul\ae\ for $u(x)$ given in \eqref{eq:yzu-define} are particularly useful to us because they do not require differentiation; however a more compact formula obtained by combining the definition of $y(x)$ in \eqref{eq:yzu-define} with the first differential equation in \eqref{eq:PIV-system} is simply (see \cite[Eqn.\@ (5.1.13)]{FokasIKN:2006})
\eq
u(x)=-2x-\frac{\dd}{\dd x}\log(Y^\infty_{1,12}(x)).
\endeq
Note also that up to a constant factor, the alternate representation of $u(x)$ given in \eqref{eq:yzu-define} above (see also \eqref{eq:u-ucirc}) is the reciprocal of the formula for extracting the solution of the Painlev\'e-III equation from its Riemann-Hilbert problem (see \cite[Eqn.\@ 18]{BothnerMS18}, which corrects a corresponding formula in \cite[Theorem 5.4]{FokasIKN:2006}).
\end{remark}

Most of these results can be found in \cite[Chapter 5, Section 1]{FokasIKN:2006}, however another observation about Riemann-Hilbert Problem~\ref{rhp:general} that is not as well-known but that we will find extremely useful is that it contains also the solution of \emph{another} Painlev\'e-IV equation:
\begin{proposition}
Suppose that $\Theta_0\neq 0$, that Riemann-Hilbert Problem~\ref{rhp:general} has a solution for $x\in\mathbb{C}\setminus\mathcal{D}$, where $\mathcal{D}$ is a discrete set, and that the function $u(x)$ given by \eqref{eq:yzu-define} is well-defined as a meromorphic function on $\mathbb{C}\setminus\mathcal{D}$.  Let $u_\tw(x)$ be defined in terms of $\mathbf{Y}(\lambda;x)$ by \eqref{eq:u-ucirc}.
%\eq
%u_\tw(x):=-2\frac{Z_{0,21}(x)Y_{1,12}(x)}{Z_{0,11}(x)}.
%\label{eq:v-define}
%\endeq
Then $u_\tw(x)$ is a meromorphic solution of the Painlev\'e-IV equation \eqref{p4} with modified parameters $(\Theta_{0,\tw},\Theta_{\infty,\tw})$ given in \eqref{eq:Baecklund-3-to-1}, 
%\eq
%\Theta_{0,\tw}:=-\tfrac{1}{2}(\Theta_0+\Theta_\infty)\quad\text{and}\quad
%\Theta_{\infty,\tw}:=\tfrac{3}{2}\Theta_0-\tfrac{1}{2}\Theta_\infty+1
%\label{eq:twisted-Thetas}
%\endeq
and $u_\tw(x)$ is explicitly related to $u(x)$ by the B\"acklund transformation indicated also in \eqref{eq:Baecklund-3-to-1}.
%\eq
%u_\tw(x)=\frac{u'(x)}{2u(x)}-\frac{2\Theta_0}{u(x)}-x-\frac{1}{2}u(x).
%\label{eq:u-to-v}
%\endeq
Thus the same Riemann-Hilbert problem encodes explicitly a solution $(\Theta_0,\Theta_\infty,u(x))$ and its image under the symmetry $\mathcal{S}_\tw$.
\label{prop:vBaecklund}
\end{proposition}
\begin{proof}
Without the hypothesis that $\Theta_0\neq 0$, by Proposition~\ref{prop:DiffEqs}, the function $u(x)$ given by the alternate expression in \eqref{eq:yzu-define} is 
a meromorphic solution of the Painlev\'e-IV equation \eqref{p4}.  However since $\Theta_0\neq 0$ it follows easily that \eqref{p4} does not admit the identically vanishing solution, so it follows in particular that $Y^0_{0,11}(x)$ does not vanish identically.  Therefore, $u_\tw(x)$ is indeed a meromorphic function, and comparing with the alternate definitions of $z(x)$ and $u(x)$ in \eqref{eq:yzu-define} one
sees that $u_\tw(x)=-2z(x)/u(x)$.  Eliminating $z(x)$ using the last equation in the system \eqref{eq:PIV-system} yields the B\"acklund transformation formula given in \eqref{eq:Baecklund-3-to-1}.  It is then straightforward to deduce the Painlev\'e-IV equation satisfied by $u_\tw(x)$ from that satisfied by $u(x)$.
\end{proof}

\begin{corollary}
Under the same additional conditions as in Corollary~\ref{cor:Schwarz}, it follows that 
$y(x^*)=y(x)^*$, $z(x^*)=z(x)^*$, $u(x^*)=u(x)^*$, and $u_\tw(x^*)=u_\tw(x)^*$.
\end{corollary}

\begin{remark}
The Painlev\'e-IV equation involves the parameter $\Theta_0$ only in the form of its square $\Theta_0^2$, but the Lax pair and Riemann-Hilbert problem break the symmetry of $\Theta_0\mapsto -\Theta_0$.  However, symmetry in the Lax pair is easily restored with the use of the quantity $\widehat{z}(x):=z(x)-\Theta_0$ in place of $z(x)$.  Then the matrices $\mathbf{X}$ and $\mathbf{\Lambda}$ are written in terms of $u$, $y$, $\widehat{z}$, $\Theta_\infty$, and $\Theta_0^2$, but $\Theta_0$ alone does not appear.  From \eqref{eq:yzu-define} one sees that the value of $\Theta_0$ is not required to express $\widehat{z}(x)$ in terms of the solution of the Riemann-Hilbert problem, and making the substitution $z(x)=\Theta_0+\widehat{z}(x)$ shows that the first order system on $y$, $\widehat{z}$, and $u$ implied by \eqref{eq:PIV-system} only involves $\Theta_0$ via its square.  Finally, making the substitution $\mathbf{Y}(\lambda;x)\mapsto \mathbf{Y}(\lambda;x)\ee^{(\frac{1}{2}\lambda^2+x\lambda)\sigma_3}\ii\sigma_1\ee^{-(\frac{1}{2}\lambda^2+x\lambda)\sigma_3}$ for $|\lambda|<1$ one obtains an equivalent Riemann-Hilbert problem in which $\Theta_0$ is replaced by $-\Theta_0$, while $Y^0_{0,11}(x)Y^0_{0,12}(x)\mapsto -Y^0_{0,12}(x)Y^0_{0,11}(x)$ leaving the alternate expression for $u(x)$ in \eqref{eq:yzu-define} (see also \eqref{eq:u-ucirc}) invariant as well.  However, the expression for $u_\tw(x)$ in \eqref{eq:u-ucirc} is \emph{not} invariant, and indeed the target parameters $\Theta_{0,\tw}^2$ and $\Theta_{\infty,\tw}$ for which $u_\tw(x)$ solves \eqref{p4} depend on $\Theta_0$ and not its square.
\label{rem:Theta0-minus}
\end{remark}

Note that making the substitution $\mathbf{\Psi}(\lambda,x)=\mathbf{Y}(\lambda;x)\ee^{(\frac{1}{2}\lambda^2+x\lambda)\sigma_3}$ with either the expansion \eqref{eq:Y-infty-expand} or \eqref{eq:Y-zero-expand} into the Lax equation $\mathbf{\Psi}_\lambda=\mathbf{\Lambda}\mathbf{\Psi}$ with coefficient matrix given by \eqref{eq:JM-A}, and separating out the powers of $\lambda$ yields from each expansion an infinite hierarchy of algebraic identities such as (from \eqref{eq:Y-zero-expand}):
%where
%$\mathbf{A}$ is the Laurent polynomial $\mathbf{A}=\lambda\sigma_3 + \mathbf{A}_0(x) +\lambda^{-1}\mathbf{A}_1(x)$ with coefficients (cf., \eqref{eq:JM-A})
%\eq
%\begin{split}
%\mathbf{A}_0(x)&=x\sigma_3 + \bpm 0 & y(x)\\ 2y(x)^{-1}(z(x)-\Theta_0-\Theta_\infty) & 0\epm\\
%\mathbf{A}_1(x)&=(\Theta_0-z(x))\sigma_3 + \bpm 0 & -\tfrac{1}{2}u(x)y(x)\\
%2y(x)^{-1}u(x)^{-1}z(x)(z(x)-2\Theta_0) & 0\epm,
%\end{split}
%\label{eq:A0A1}
%\endeq
%yields infinitely many algebraic identities of the form
%\begin{multline}
%-k\mathbf{Y}_k(x)-\Theta_\infty\mathbf{Y}_k(x)\sigma_3+[\mathbf{Y}_{k+2}(x),\sigma_3]+x\mathbf{Y}_{k+1}(x)\sigma_3-\mathbf{A}_0(x)\mathbf{Y}_{k+1}(x)-\mathbf{A}_1(x)\mathbf{Y}_k(x)=\mathbf{0},\\
%k=1,2,3,\dots.
%\label{eq:coefficient-identities}
%\end{multline}
%Similarly, using \eqref{eq:Y-zero-expand} instead of \eqref{eq:Y-infty-expand} one obtains the identities
%\eq
%\Theta_0\mathbf{Z}_0(x)\sigma_3-\mathbf{A}_1(x)\mathbf{Z}_0(x)=\mathbf{0},
%\endeq
\eq
\mathbf{Y}^0_1(x)+\Theta_0\mathbf{Y}^0_1(x)\sigma_3-\mathbf{\Lambda}_1(x)\mathbf{Y}^0_1(x) + x\mathbf{Y}^0_0(x)\sigma_3-\mathbf{\Lambda}_0(x)\mathbf{Y}^0_0(x)=\mathbf{0}.
\label{eq:zero-1-identity}
\endeq
%and then infinitely many identities
%\begin{multline}
%(k+2)\mathbf{Z}_{k+2}(x)+\Theta_0\mathbf{Z}_{k+2}(x)\sigma_3-\mathbf{A}_1(x)\mathbf{Z}_{k+2}(x)+x\mathbf{Z}_{k+1}(x)\sigma_3-\mathbf{A}_0(x)\mathbf{Z}_{k+1}(x)+[\mathbf{Z}_k(x),\sigma_3]=\mathbf{0},\\
%k=0,1,2,\dots.
%\end{multline}
%\textcolor{red}{Of the above formul\ae\, probably only \eqref{eq:zero-1-identity} is used below.  Consider omitting all but that one.}

\subsection{Isomonodromic Schlesinger transformations}
\label{sec:Schlesinger}
The parameter lattices $\Lambda_\mathrm{gH}^{[3]+}$ and $\Lambda_\mathrm{gO}$ are both generated from a given point by linear combinations over integers of diagonal lattice vectors $(\tfrac{1}{2},\pm\tfrac{1}{2})$.  If we think of fixing the monodromy data $\mathbf{V}$ in Riemann-Hilbert Problem~\ref{rhp:general} (up to a sign for $\mathbf{V}_{4,3}$) and letting the parameters $(\Theta_0,\Theta_\infty)$ vary, then it is possible to explicitly relate the solutions for two instances of this problem when the parameters differ by $(\tfrac{1}{2},\pm\tfrac{1}{2})$, via a certain left-multiplier (i.e., gauge transformation of the Lax pair) called a \emph{Schlesinger transformation}.  Since the monodromy data is unchanged, these transformations are \emph{isomonodromic}.  When we use formul\ae\ such as \eqref{eq:u-ucirc} to extract $u(x)$ and $u_\tw(x)$ from these related problems, we obtain explicit relations between these functions for related parameter values, which are \emph{isomonodromic B\"acklund transformations}.

\subsubsection{Basic Schlesinger transformations}
\label{sec:sub-Schlesinger}
Suppose $\mathbf{Y}(\lambda;x)$ is the solution of Riemann-Hilbert Problem~\ref{rhp:general} for parameters $(\Theta_0,\Theta_\infty)$.  There are four basic transformations $\mathcal{T}$ of $\mathbf{Y}(\lambda;x)$ we will develop, all of which are based on the same kind of gauge transformation formula:
\eq
\mathcal{T}\mathbf{Y}(\lambda;x):=\mathbf{G}(\lambda;x)\mathbf{Y}(\lambda;x),\quad
\mathbf{G}(\lambda;x)=\mathbf{G}^+(x)\lambda^{1/2} +\mathbf{G}^-(x)\lambda^{-1/2},
\label{eq:gauge-general}
\endeq
where the power functions refer to the principal branches, cut along $\lambda<0$, and the matrix coefficients $\mathbf{G}^\pm(x)$ are to be determined so that $\mathcal{T}\mathbf{Y}(\lambda;x)$ solves a Riemann-Hilbert problem closely related to Riemann-Hilbert Problem~\ref{rhp:general}.
Indeed, from \eqref{eq:gauge-general} it is straightforward to check that domains of analyticity of $\mathbf{Y}(\lambda;x)$ and $\mathcal{T}\mathbf{Y}(\lambda;x)$ agree, and that the jump conditions are preserved except on $\Sigma_{4,3}$, where $\mathbf{V}_{4,3}$ is replaced with $-\mathbf{V}_{4,3}$. 
So $\mathbf{G}^\pm(x)$ are to be chosen so that the conditions specifying the behavior of the solution at $\lambda=\infty$ and $\lambda=0$ hold in some form. 

To replace $(\Theta_0,\Theta_\infty)$ by $(\Theta_0+\tfrac{1}{2},\Theta_\infty+\tfrac{1}{2})$, we require that $\mathbf{Y}_\nearrow(\lambda;x)=\mathcal{T}\mathbf{Y}(\lambda;x)$ satisfy the conditions
\eq
\begin{split}
\mathbf{Y}_\nearrow(\lambda;x)\lambda^{(\Theta_\infty+\tfrac{1}{2})\sigma_3}&=\mathbb{I} + \mathbf{Y}^\infty_{\nearrow,1}(x)\lambda^{-1} + \bo(\lambda^{-2}),\quad\lambda\to\infty\\
\mathbf{Y}_\nearrow(\lambda;x)\lambda^{-(\Theta_0+\tfrac{1}{2})\sigma_3}&=\mathbf{Y}^0_{\nearrow,0}(x)+ \bo(\lambda),\quad\lambda\to 0
\end{split}
\label{eq:conditions-ne}
\endeq
for some matrices $\mathbf{Y}^\infty_{\nearrow,1}(x)$ and $\mathbf{Y}^0_{\nearrow,0}(x)$.  Using the expansions \eqref{eq:Y-infty-expand} and \eqref{eq:Y-zero-expand} it is straightforward to see that, assuming the matrix element $Y^0_{0,21}(x)$ in \eqref{eq:Y-zero-expand} is a meromorphic function that does not vanish identically, these relations will hold if and only if
\eq
\mathbf{G}^+(x)=\mathbf{G}^+_\nearrow:=\bpm 0&0\\0 & 1\epm\quad\text{and}\quad
\mathbf{G}^-(x)=\mathbf{G}^-_\nearrow(x):=\bpm 1 & -Y^0_{0,21}(x)^{-1}Y^0_{0,11}(x)\\-Y^\infty_{1,21}(x) & Y^\infty_{1,21}(x)Y^0_{0,21}(x)^{-1}Y^0_{0,11}(x)\epm,
\label{eq:gauge-ne-coeffs}
\endeq
defining $\mathcal{T}=\mathcal{T}_\nearrow$ as a gauge transformation that exists except at the isolated zeros of $Y^0_{0,21}(x)$.  Note that, using \eqref{eq:yzu-define}, we may write $\mathbf{G}^-_\nearrow(x)$ in the form
\eq
\mathbf{G}^-_\nearrow(x)=\bpm 1 & \tfrac{1}{2}u(x)y(x)z(x)^{-1}\\(\Theta_0+\Theta_\infty-z(x))y(x)^{-1} &
\tfrac{1}{2}(\Theta_0+\Theta_\infty-z(x))u(x)z(x)^{-1}\epm.
\label{eq:D-ne}
\endeq
It is straightforward to check that subject to \eqref{eq:gauge-ne-coeffs}, $\det(\mathbf{G}(\lambda;x))=1$, so this transformation preserves determinants.  The gauge matrix $\mathbf{G}(\lambda;x)$ in this case is denoted $\mathbf{R}_3$ in \cite{FokasMA88}.  See also \cite[Section 2.1]{BassomCH95}.

To replace $(\Theta_0,\Theta_\infty)$ by $(\Theta_0-\tfrac{1}{2},\Theta_\infty-\tfrac{1}{2})$, we 
require that $\mathbf{Y}_\swarrow(\lambda;x)=\mathcal{T}\mathbf{Y}(\lambda;x)$ satisfy 
\eq
\begin{split}
\mathbf{Y}_\swarrow(\lambda;x)\lambda^{(\Theta_\infty-\tfrac{1}{2})\sigma_3}&=\mathbb{I} + \mathbf{Y}^\infty_{\swarrow,1}(x)\lambda^{-1} + \bo(\lambda^{-2}),\quad\lambda\to\infty\\
\mathbf{Y}_\swarrow(\lambda;x)\lambda^{-(\Theta_0-\tfrac{1}{2})\sigma_3}&=\mathbf{Y}^0_{\swarrow,0}(x)+ \bo(\lambda),\quad\lambda\to 0
\end{split}
\label{eq:conditions-sw}
\endeq
for some matrices $\mathbf{Y}^\infty_{\swarrow,1}(x)$ and $\mathbf{Y}^0_{\swarrow,0}(x)$.  Again using the expansions \eqref{eq:Y-infty-expand} and \eqref{eq:Y-zero-expand}, assuming now that $Y^0_{0,12}(x)$ does not vanish identically, the above conditions will hold if and only if
\eq
\mathbf{G}^+(x)=\mathbf{G}^+_\swarrow:=\bpm 1&0\\0 &0\epm\quad\text{and}\quad
\mathbf{G}^-(x)=\mathbf{G}^-_\swarrow(x):=\bpm Y^\infty_{1,12}(x)Y^0_{0,12}(x)^{-1}Y^0_{0,22}(x) &-Y^\infty_{1,12}(x)\\
-Y^0_{0,12}(x)^{-1}Y^0_{0,22}(x) & 1\epm,
\label{eq:gauge-sw-coeffs}
\endeq
defining $\mathcal{T}=\mathcal{T}_\swarrow$ as a gauge transformation that exists except at the isolated zeros of $Y^0_{0,12}(x)$ and preserves determinants.  Using \eqref{eq:yzu-define} we can also write
\eq
\mathbf{G}^-_\swarrow(x)=\bpm (z(x)-2\Theta_0)u(x)^{-1} & \tfrac{1}{2}y(x)\\
2(z(x)-2\Theta_0)u(x)^{-1}y(x)^{-1} & 1\epm.
\label{eq:D-sw}
\endeq
The gauge matrix $\mathbf{G}(\lambda;x)$ in this case is denoted $\mathbf{R}_4$ in \cite{FokasMA88}.

To replace $(\Theta_0,\Theta_\infty)$ with $(\Theta_0+\tfrac{1}{2},\Theta_\infty-\tfrac{1}{2})$ we insist that $\mathbf{Y}_\searrow(\lambda;x)=\mathcal{T}\mathbf{Y}(\lambda;x)$ satisfy the conditions
\eq
\begin{split}
\mathbf{Y}_\searrow(\lambda;x)\lambda^{(\Theta_\infty-\tfrac{1}{2})\sigma_3}&=\mathbb{I}+\mathbf{Y}^\infty_{\searrow,1}(x)\lambda^{-1}+\bo(\lambda^{-2}),\quad\lambda\to\infty\\
\mathbf{Y}_\searrow(\lambda;x)\lambda^{-(\Theta_0+\tfrac{1}{2})\sigma_3}&=\mathbf{Y}^0_{\searrow,0}(x)+\bo(\lambda),\quad\lambda\to 0
\end{split}
\label{eq:conditions-se}
\endeq
for some matrices $\mathbf{Y}^\infty_{\searrow,1}(x)$ and $\mathbf{Y}^0_{\searrow,0}(x)$.  Assuming that $Y^0_{0,11}(x)$ does not vanish identically, these conditions hold if and only if
\eq
\mathbf{G}^+(x)=\mathbf{G}^+_\searrow:=\bpm 1&0\\0&0\epm\quad\text{and}\quad
\mathbf{G}^-(x)=\mathbf{G}^-_\searrow:=\bpm Y^\infty_{1,12}(x)Y^0_{0,11}(x)^{-1}Y^0_{0,21}(x) & -Y^\infty_{1,12}(x)\\
-Y^0_{0,11}(x)^{-1}Y^0_{0,21}(x) & 1\epm,
\label{eq:gauge-se-coeffs}
\endeq
defining $\mathcal{T}=\mathcal{T}_\searrow$ as a gauge transformation existing except at the isolated zeros of $Y^0_{0,11}(x)$ and preserving determinants.  From \eqref{eq:yzu-define} we also have
\eq
\mathbf{G}^-_\searrow(x)=\bpm
z(x)u(x)^{-1} & \tfrac{1}{2}y(x)\\ 2z(x)u(x)^{-1}y(x)^{-1} & 1
\epm.
\label{eq:D-se}
\endeq
The gauge matrix $\mathbf{G}(\lambda;x)$ for this case is denoted $\mathbf{R}_2$ in \cite{FokasMA88}.

Finally, to replace $(\Theta_0,\Theta_\infty)$ with $(\Theta_0-\tfrac{1}{2},\Theta_0+\tfrac{1}{2})$ we insist that $\mathbf{Y}_\nwarrow(\lambda;x)=\mathcal{T}\mathbf{Y}(\lambda;x)$ satisfy the conditions
\eq
\begin{split}
\mathbf{Y}_\nwarrow(\lambda;x)\lambda^{(\Theta_\infty+\tfrac{1}{2})\sigma_3}&=\mathbb{I}+\mathbf{Y}^\infty_{\nwarrow,1}(x)\lambda^{-1}+\bo(\lambda^{-2}),\quad\lambda\to\infty\\
\mathbf{Y}_\nwarrow(\lambda;x)\lambda^{-(\Theta_0-\tfrac{1}{2})\sigma_3}&=\mathbf{Y}^0_{\nwarrow,0}(x)+\bo(\lambda),\quad\lambda\to 0
\end{split}
\label{eq:conditions-nw}
\endeq
for some matrices $\mathbf{Y}^\infty_{\nwarrow,1}(x)$ and $\mathbf{Y}^0_{\nwarrow,0}(x)$.  Assuming that $Y^0_{0,22}(x)$ does not vanish identically, these conditions hold if and only if
\eq
\mathbf{G}^+(x)=\mathbf{G}^+_\nwarrow=\bpm 0 & 0\\0 & 1\epm\quad\text{and}\quad
\mathbf{G}^-(x)=\mathbf{G}^-_\nwarrow(x)=\bpm 1 & -Y^0_{0,22}(x)^{-1}Y^0_{0,12}(x)\\
-Y^\infty_{1,21}(x) & Y^\infty_{1,21}(x)Y^0_{0,22}(x)^{-1}Y^0_{0,12}(x)\epm,
\label{eq:gauge-nw-coeffs}
\endeq
defining $\mathcal{T}=\mathcal{T}_\nwarrow$ as a gauge transformation existing except at the isolated zeros of $Y^0_{0,22}(x)$ and preserving determinants.  From \eqref{eq:yzu-define} we can write
\eq
\mathbf{G}^-_\nwarrow(x)=\bpm 1 & \tfrac{1}{2}u(x)y(x)(z(x)-2\Theta_0)^{-1}\\
(\Theta_0+\Theta_\infty-z(x))y(x)^{-1} & \tfrac{1}{2}u(x)(\Theta_0+\Theta_\infty-z(x))(z(x)-2\Theta_0)^{-1}
\epm.
\label{eq:D-nw}
\endeq
The gauge matrix $\mathbf{G}(\lambda;x)$ for this case is denoted $\mathbf{R}_1$ in \cite{FokasMA88}.

We have therefore established the following:
\begin{proposition}
Suppose that $\mathbf{Y}(\lambda;x)$ is the solution of Riemann-Hilbert Problem~\ref{rhp:general}.  
\begin{itemize}
\item
If $Y^0_{0,21}(x)$ does not vanish identically, then $\mathbf{Y}_\nearrow(\lambda;x)$ defined by \eqref{eq:gauge-general} with \eqref{eq:gauge-ne-coeffs} solves an analogous Riemann-Hilbert problem in which parameters $(\Theta_0,\Theta_\infty)$ are replaced with $(\Theta_0+\tfrac{1}{2},\Theta_\infty+\tfrac{1}{2})$ and the sign of the matrix $\mathbf{V}_{4,3}$ is changed.
\item
If $Y^0_{0,12}(x)$ does not vanish identically, then $\mathbf{Y}_\swarrow(\lambda;x)$ defined by \eqref{eq:gauge-general} with \eqref{eq:gauge-sw-coeffs} solves an analogous Riemann-Hilbert problem in which parameters $(\Theta_0,\Theta_\infty)$ are replaced by $(\Theta_0-\tfrac{1}{2},\Theta_\infty-\tfrac{1}{2})$ and the sign of the matrix $\mathbf{V}_{4,3}$ is changed.  
\item
If $Y^0_{0,11}(x)$ does not vanish identically, then $\mathbf{Y}_\searrow(\lambda;x)$ defined by \eqref{eq:gauge-general} with \eqref{eq:gauge-se-coeffs} solves an analogous Riemann-Hilbert problem in which parameters $(\Theta_0,\Theta_\infty)$ are replaced by $(\Theta_0+\tfrac{1}{2},\Theta_\infty-\tfrac{1}{2})$ and the sign of the matrix $\mathbf{V}_{4,3}$ is changed.  
\item
If $Y^0_{0,22}(x)$ does not vanish identically, then $\mathbf{Y}_\nwarrow(\lambda;x)$ defined by \eqref{eq:gauge-general} with \eqref{eq:gauge-nw-coeffs} solves an analogous Riemann-Hilbert problem in which parameters $(\Theta_0,\Theta_\infty)$ are replaced by $(\Theta_0-\tfrac{1}{2},\Theta_\infty+\tfrac{1}{2})$ and the sign of the matrix $\mathbf{V}_{4,3}$ is changed.  
\end{itemize}
\label{prop:Schlesinger}
\end{proposition}
The arrow subscript notation is doubly useful.  On one hand, it indicates the direction of the step in the $(\Theta_0,\Theta_\infty)$-plane.  On the other hand, if one centers the arrow on the matrix $\mathbf{Y}^0_0(x)$ as it is usually written in terms of its elements on the page, the arrow begins at the matrix element required to not vanish identically for the transformation to be defined.  We can easily give conditions on the parameters $(\Theta_0,\Theta_\infty)$ sufficient to guarantee existence of the transformation in each case:
%
%\textcolor{red}{Some parts of this are needed to ensure we can apply the transformations to map out all the rationals in sector $W^{[3]+}$.}
\begin{proposition}
Suppose that Riemann-Hilbert Problem~\ref{rhp:general} has a solution for $x\in\mathbb{C}\setminus\mathcal{D}$, where $\mathcal{D}$ is a discrete set.  
\begin{itemize}
\item
The Schlesinger transformations $\mathbf{Y}(\lambda;x)\mapsto \mathbf{Y}_\nearrow(\lambda;x)$ and $\mathbf{Y}(\lambda;x)\mapsto\mathbf{Y}_\swarrow(\lambda;x)$ are well-defined provided that $\Theta_0(\Theta_\infty+\Theta_0)\neq 0$.
\item
The Schlesinger transformations $\mathbf{Y}(\lambda;x)\mapsto \mathbf{Y}_\searrow(\lambda;x)$ and $\mathbf{Y}(\lambda;x)\mapsto\mathbf{Y}_\nwarrow(\lambda;x)$ are well-defined provided that $\Theta_0(\Theta_\infty-\Theta_0)\neq 0$.
\end{itemize}
%
%none of the meromorphic functions $Y^{(m,n)}_{1,12}(x)$, $Y^{(m,n)}_{1,21}(x)$, $Z^{(m,n)}_{0,12}(x)$, nor $Z^{(m,n)}_{0,21}(x)$ vanishes identically on $\mathbb{C}\setminus\mathcal{D}^{(m,n)}$.
\label{prop:nonzero}
\end{proposition}

\begin{remark}
The conditions on $(\Theta_0,\Theta_\infty)$ in Proposition~\ref{prop:nonzero} imply the corresponding conditions on the elements of $\mathbf{Y}^0_0(x)$ in Proposition~\ref{prop:Schlesinger} but not necessarily vice-versa.  We will encounter a situation in which $\Theta_\infty=\Theta_0$ but the Schlesinger transformation $\mathbf{Y}(\lambda;x)\mapsto \mathbf{Y}_\searrow(\lambda;x)$ exists nonetheless, because $Y^0_{0,11}(x)$ does not vanish identically.  The utility of Proposition~\ref{prop:nonzero} lies in the simplicity of its conditions, which do not depend on the choice of solution of Painlev\'e-IV for the given parameters $(\Theta_0,\Theta_\infty)$.
\label{rem:weaker}
\end{remark}

\begin{proof}
Suppose first that the transformation $\mathbf{Y}(\lambda;x)\mapsto\mathbf{Y}_\swarrow(\lambda;x)$ is undefined.  Therefore, by Proposition~\ref{prop:Schlesinger} $Y^0_{0,12}(x)\equiv 0$ on some open subset of 
$(\lambda,x)\in (\mathbb{C}\setminus\Sigma)\times (\mathbb{C}\setminus\mathcal{D})$.  Since $\det(\mathbf{Y}^0_0(x))\equiv 1$ by the first statement of Proposition~\ref{prop:RHP-basic}, it then follows that also $Y^0_{0,11}(x)Y^0_{0,22}(x)\equiv 1$ so the matrix $\mathbf{\Lambda}_1(x)$ as given by the second line of \eqref{eq:M-define} can be written in the form
\eq
\mathbf{\Lambda}_1(x)=\bpm \Theta_0 & 0\\ V(x) & -\Theta_0\epm
\endeq
for some analytic function $V(x)$.  Therefore, the Lax equation \eqref{eq:Lax-x} takes the form
\eq
\frac{\partial\mathbf{\Psi}}{\partial x}(\lambda,x)=\bpm\lambda & U(x)\\W(x) & -\lambda\epm
\mathbf{\Psi}(\lambda,x)
\label{eq:Lax-x-case-2}
\endeq
for some analytic functions $U(x)$ and $W(x)$, while the Lax equation \eqref{eq:Lax-lambda} becomes
\eq
\frac{\partial\mathbf{\Psi}}{\partial\lambda}(\lambda,x)=\bpm \lambda + x + \lambda^{-1}\Theta_0 & U(x)\\W(x)+\lambda^{-1}V(x) & -\lambda-x-\lambda^{-1}\Theta_0\epm\mathbf{\Psi}(\lambda,x).
\label{eq:Lax-lambda-case-2}
\endeq
Compatibility of the equations \eqref{eq:Lax-x-case-2} and \eqref{eq:Lax-lambda-case-2} implies that either $\Theta_0=0$ or $U(x)\equiv 0$.  If $\Theta_0\neq 0$, the alternative $U(x)\equiv 0$ implies, via \eqref{eq:Lax-x-case-2} and \eqref{eq:Lax-lambda-case-2} that 
\eq
\Psi_{11}(\lambda,x)=c\ee^{\frac{1}{2}\lambda^2 +x\lambda}\lambda^{\Theta_0}, 
\endeq
where $c$ is a constant that can take different values in each of the five components of $\mathbb{C}\setminus\Sigma$.  Consider letting $\lambda\to\infty$ in any one of the four unbounded components of $\mathbb{C}\setminus\Sigma$.  Then from the normalization condition in Riemann-Hilbert Problem~\ref{rhp:general} and the relation $\mathbf{\Psi}(\lambda,x)=\mathbf{Y}(\lambda;x)\ee^{(\frac{1}{2}\lambda^2+x\lambda)\sigma_3}$ we arrive at a contradiction unless $\Theta_0+\Theta_\infty=0$.  

Now suppose instead that the Schlesinger transformation $\mathbf{Y}(\lambda;x)\mapsto\mathbf{Y}_\nearrow(\lambda;x)$ is undefined, which by Proposition~\ref{prop:Schlesinger} means that $Y^0_{0,21}(x)\equiv 0$. Following similar reasoning as above, we then arrive at \eqref{eq:Lax-x-case-2} and \eqref{eq:Lax-lambda-case-2} in which the coefficient matrices are replaced by their transposes; compatibility implies again either $\Theta_0=0$ or $U(x)\equiv 0$, which in turn implies that $\Psi_{22}(\lambda,x)=c\ee^{-(\frac{1}{2}\lambda^2+x\lambda)}\lambda^{-\Theta_0}$ for a different constant $c$ in each component of $\mathbb{C}\setminus\Sigma$.  Once again, consistency with the normalization condition in Riemann-Hilbert Problem~\ref{rhp:general} leads to a contradiction unless $\Theta_0+\Theta_\infty=0$.

To obtain the corresponding results for the Schlesinger transformations $\mathbf{Y}(\lambda;x)\mapsto\mathbf{Y}_\searrow(\lambda;x)$ and $\mathbf{Y}(\lambda;x)\mapsto\mathbf{Y}_\nwarrow(\lambda;x)$, we can apply similar reasoning as in Remark~\ref{rem:Theta0-minus} to first replace $\Theta_0$ with $-\Theta_0$ at the cost of essentially swapping the columns of the matrix $\mathbf{Y}^0_0(x)$.  It then follows from the above arguments that these transformations will be defined unless either $\Theta_0=0$ or $\Theta_\infty-\Theta_0=0$.  
\end{proof}

\subsubsection{Corresponding B\"acklund transformations}
\label{sec:sub-Baecklund}
Now we suppose that $\mathbf{Y}(\lambda;x)$ solves Riemann-Hilbert Problem~\ref{rhp:general} and that $u(x)$ given in terms of the solution by \eqref{eq:yzu-define} is well-defined.  According to Proposition~\ref{prop:DiffEqs}, $u(x)$ is a meromorphic solution of the Painlev\'e-IV equation in the form \eqref{p4} for parameters $(\Theta_0,\Theta_\infty)$.  We now deduce from the Schlesinger transformations summarized in Proposition~\ref{prop:Schlesinger} the corresponding solutions $u_\nearrow(x)$, $u_\swarrow(x)$, $u_\searrow(x)$, and $u_\nwarrow(x)$ of \eqref{p4} for the modified parameters indicated in Proposition~\ref{prop:Schlesinger}.

If $Y^0_{0,21}(x)$ is not identically zero, then $\mathbf{Y}_\nearrow(\lambda;x)$ exists, and it generates a solution of \eqref{p4} for parameters $(\Theta_0+\tfrac{1}{2},\Theta_\infty+\tfrac{1}{2})$ given by (cf.\@ \eqref{eq:yzu-define})
\eq
u_\nearrow(x):=-2(\Theta_0+\tfrac{1}{2})\frac{Y^0_{\nearrow,0,11}(x)Y^0_{\nearrow,0,12}(x)}{Y^\infty_{\nearrow,1,12}(x)}
\endeq
provided the latter expression is determinate.  The formul\ae\ for the matrices $\mathbf{Y}^\infty_{\nearrow,1}(x)$ and $\mathbf{Y}^0_{\nearrow,0}(x)$ appearing in \eqref{eq:conditions-ne} are
\eq
\begin{split}
\mathbf{Y}^\infty_{\nearrow,1}(x)&=\mathbf{G}^+_\nearrow\mathbf{Y}^\infty_2(x)\bpm 1&0\\0&0\epm +
\mathbf{G}^+_\nearrow\mathbf{Y}^\infty_1(x)\bpm 0 & 0\\0 & 1\epm +\mathbf{G}^-_\nearrow(x)\mathbf{Y}^\infty_1(x)\bpm 1&0\\ 0 & 0\epm +\mathbf{G}^-_\nearrow(x)\bpm 0 & 0\\0 & 1\epm,\\
\mathbf{Y}^0_{\nearrow,0}(x)&=\mathbf{G}^+_\nearrow\mathbf{Y}^0_0(x)\bpm 1 & 0\\0 & 0\epm 
+\mathbf{G}^-_\nearrow(x)\mathbf{Y}^0_1(x)\bpm 1 & 0\\0 & 0\epm+\mathbf{G}^-_\nearrow(x)\mathbf{Y}^0_0(x)\bpm 0 & 0\\0 & 1\epm.
\end{split}
\endeq
Using \eqref{eq:gauge-ne-coeffs}--\eqref{eq:D-ne} in these, along with the identity \eqref{eq:zero-1-identity} to eliminate elements of the matrix $\mathbf{Y}^0_1(x)$ and the definitions \eqref{eq:yzu-define} as well as the first-order system \eqref{eq:PIV-system}, we may express $u_\nearrow(x)$ explicitly in terms of $u(x)$ and $u'(x)$ as:
\eq
u_\nearrow(x)=\frac{16\Theta_0^2+8(\Theta_0+\Theta_\infty)u(x)^2-4x^2u(x)^2-4xu(x)^3-u(x)^4-8\Theta_0u'(x)+u'(x)^2}{2u(x)(4\Theta_0+2xu(x)+u(x)^2-u'(x))}.
\label{eq:u-ne}
\endeq

If $Y^0_{0,12}(x)$ does not vanish identically, we can apply similar reasoning to extract a solution $u_\swarrow(x)$ of Painlev\'e-IV \eqref{p4} for parameters $(\Theta_0-\tfrac{1}{2},\Theta_\infty-\tfrac{1}{2})$ from $u(x)$.  The starting point is
\eq
u_\swarrow(x):=-2(\Theta_0-\tfrac{1}{2})\frac{Y^0_{\swarrow,0,11}(x)Y^0_{\swarrow,0,12}(x)}{Y^\infty_{\swarrow,1,12}(x)},
\endeq
assuming this expression is determinate, where the matrices $\mathbf{Y}^\infty_{\swarrow,1}(x)$ and $\mathbf{Y}^0_{\swarrow,0}(x)$ in \eqref{eq:conditions-sw} are given by
\eq
\begin{split}
\mathbf{Y}^\infty_{\swarrow,1}(x)&= \mathbf{G}^+_\swarrow\mathbf{Y}^\infty_2(x)\bpm 0&0\\0&1\epm+\mathbf{G}^+_\swarrow\mathbf{Y}^\infty_1(x)\bpm 1&0\\0 & 0\epm + \mathbf{G}^-_\swarrow(x)\mathbf{Y}^\infty_1(x)\bpm 0&0\\0&1\epm
+\mathbf{G}^-_\swarrow\bpm 1&0\\0&0\epm,\\
\mathbf{Y}^0_{\swarrow,0}(x)&= \mathbf{G}^+_\swarrow\mathbf{Y}^0_0(x)\bpm 0&0\\0&1\epm + 
\mathbf{G}^-_\swarrow(x)\mathbf{Y}^0_1(x)\bpm 0&0\\0&1\epm +\mathbf{G}^-_\swarrow(x)\mathbf{Y}^0_0(x)
\bpm 1&0\\0&0\epm.
\end{split}
\endeq
Using \eqref{eq:gauge-sw-coeffs}--\eqref{eq:D-sw} in these, together with \eqref{eq:yzu-define}, \eqref{eq:PIV-system}, and \eqref{eq:zero-1-identity}, we find
\eq
u_\swarrow(x)=\frac{16\Theta_0^2+8(\Theta_0+\Theta_\infty-1)u(x)^2-4x^2u(x)^2-4xu(x)^3-u(x)^4+8\Theta_0u'(x)+u'(x)^2}{2u(x)(4\Theta_0+2xu(x)+u(x)^2+u'(x))}.
\label{eq:u-sw}
\endeq

If $Y^0_{0,11}(x)$ does not vanish identically, then we can extract from $\mathbf{Y}_\searrow(\lambda;x)$ a solution $u_\searrow(x)$ of \eqref{p4} for parameters $(\Theta_0+\tfrac{1}{2},\Theta_\infty-\tfrac{1}{2})$ by starting from the expression
\eq
u_\searrow(x):=-2(\Theta_0+\tfrac{1}{2})\frac{Y^0_{\searrow,0,11}(x)Y^0_{\searrow,0,12}(x)}{Y^\infty_{\searrow,1,12}(x)}
\endeq 
provided it is determinate.  The matrices $\mathbf{Y}^\infty_{\searrow,1}(x)$ and $\mathbf{Y}^0_{\searrow,0}(x)$ from \eqref{eq:conditions-se} are given by
\eq
\begin{split}
\mathbf{Y}^\infty_{\searrow,1}(x)&=\mathbf{G}^+_\searrow\mathbf{Y}^\infty_2(x)\bpm 0&0\\0&1\epm+\mathbf{G}^+_\searrow\mathbf{Y}^\infty_1(x)\bpm 1&0\\0&0\epm + \mathbf{G}^-_\searrow(x)\mathbf{Y}^\infty_1(x)\bpm 0&0\\0&1\epm +
\mathbf{G}^-_\searrow(x)\bpm 1&0\\0&1\epm,\\
\mathbf{Y}^0_{\searrow,0}(x)&=\mathbf{G}^+_\searrow\mathbf{Y}^0_0(x)\bpm 1&0\\0&0\epm + \mathbf{G}^-_\searrow(x)
\mathbf{Y}^0_1(x)\bpm 1&0\\0&0\epm + \mathbf{G}^-_\searrow(x)\mathbf{Y}^0_0(x)\bpm 0&0\\0&1\epm.
\end{split}
\endeq
Recalling \eqref{eq:gauge-se-coeffs}--\eqref{eq:D-se},
eliminating the elements of the first column of $\mathbf{Y}^0_1(x)$ using \eqref{eq:zero-1-identity}, and then using the definitions \eqref{eq:yzu-define} and the differential equations \eqref{eq:PIV-system},
we obtain
\eq
u_\searrow(x)=\frac{16\Theta_0^2+8(\Theta_\infty-\Theta_0-1)u(x)^2-4x^2u(x)^2-4xu(x)^3-u(x)^4-8\Theta_0u'(x)+u'(x)^2}{2u(x)(-4\Theta_0+2xu(x)+u(x)^2+u'(x))}.
\label{eq:u-se}
\endeq

Finally, if $Y^0_{0,22}(x)$ does not vanish identically, we can extract from $\mathbf{Y}_\nwarrow(\lambda;x)$ a solution $u_\nwarrow(x)$ of the Painlev\'e-IV equation \eqref{p4} for parameters $(\Theta_0-\tfrac{1}{2},\Theta_\infty+\tfrac{1}{2})$ in the form
\eq
u_\nwarrow(x):=-2(\Theta_0-\tfrac{1}{2})\frac{Y^0_{\nwarrow,0,11}(x)Y^0_{\nwarrow,0,12}(x)}{Y^\infty_{\nwarrow,1,12}(x)}
\endeq
provided it is determinate.  Here the matrices $\mathbf{Y}^\infty_{\nwarrow,1}(x)$ and $\mathbf{Y}^0_{\nwarrow,0}(x)$ from \eqref{eq:conditions-nw} are given by
\eq
\begin{split}
\mathbf{Y}^\infty_{\nwarrow,1}(x)&=\mathbf{G}^+_\nwarrow\mathbf{Y}^\infty_2(x)\bpm 1&0\\0&0\epm +\mathbf{G}^+_\nwarrow\mathbf{Y}^\infty_1(x)\bpm 0&0\\0&1\epm +\mathbf{G}^-_\nwarrow(x)\mathbf{Y}^\infty_1(x)\bpm 1&0\\0&0\epm +
\mathbf{G}^-_\nwarrow(x)\bpm 0&0\\0&1\epm,\\
\mathbf{Y}^0_{\nwarrow,0}(x)&=\mathbf{G}^+_\nwarrow\mathbf{Y}^0_0(x)\bpm 0&0\\0&1\epm +
\mathbf{G}^-_\nwarrow(x)\mathbf{Y}^0_1(x)\bpm 0&0\\0&1\epm +
\mathbf{G}^-_\nwarrow(x)\mathbf{Y}^0_0(x)\bpm 1&0\\0&0\epm.
\end{split}
\endeq
Recalling \eqref{eq:gauge-nw-coeffs}--\eqref{eq:D-nw}, using \eqref{eq:zero-1-identity} to eliminate the second column of $\mathbf{Y}^0_1(x)$, and appealing to the definitions \eqref{eq:yzu-define} and the differential system \eqref{eq:PIV-system} yields
\eq
u_\nwarrow(x)=\frac{16\Theta_0^2+8(\Theta_\infty-\Theta_0)u(x)^2-4x^2u(x)^2-4xu(x)^3-u(x)^4+8\Theta_0u'(x)+u'(x)^2}{2u(x)(-4\Theta_0+2xu(x)+u(x)^2-u'(x))}.
\label{eq:u-nw}
\endeq

Even if the Schlesinger transformation exists and is applied to a solution $\mathbf{Y}(\lambda;x)$ of Riemann-Hilbert Problem~\ref{rhp:general} for which $u(x)$ is well-defined, the corresponding B\"acklund transformation formula may be indeterminate.  To detect the latter issue, we may observe first that if $\Theta_0\neq 0$, the Painlev\'e-IV equation \eqref{p4} does not admit the vanishing solution $u(x)\equiv 0$, so the problem reduces question of the existence of simultaneous solutions $u(x)$ of \eqref{p4} and of the Riccati equation obtained by setting to zero the other factor in the denominator of each of the formul\ae\ \eqref{eq:u-ne}, \eqref{eq:u-sw}, \eqref{eq:u-se}, and \eqref{eq:u-nw}.  In each case, this amounts to a condition on the parameters $(\Theta_0,\Theta_\infty)$, as summarized in the following proposition.
\begin{proposition}
Let $u(x)$ be a solution of the Painlev\'e-IV equation \eqref{p4} for parameters $(\Theta_0,\Theta_\infty)$.  
\begin{itemize}
\item The B\"acklund transformation \eqref{eq:u-ne} taking $u(x)$ to $u_\nearrow(x)$ solving \eqref{p4} for shifted parameters $(\Theta_0+\tfrac{1}{2},\Theta_\infty+\tfrac{1}{2})$ is determinate provided that $\Theta_0(\Theta_\infty+\Theta_0)\neq 0$.
\item The B\"acklund transformation \eqref{eq:u-sw} taking $u(x)$ to $u_\swarrow(x)$ solving \eqref{p4} for shifted parameters $(\Theta_0-\tfrac{1}{2},\Theta_\infty-\tfrac{1}{2})$ is determinate provided that $\Theta_0(\Theta_\infty+\Theta_0-1)\neq 0$.
\item
The B\"acklund transformation \eqref{eq:u-se} taking $u(x)$ to $u_\searrow(x)$ solving \eqref{p4} for shifted parameters $(\Theta_0+\tfrac{1}{2},\Theta_\infty-\tfrac{1}{2})$ is determinate provided that $\Theta_0(\Theta_\infty-\Theta_0-1)\neq 0$.
\item 
The B\"acklund transformation \eqref{eq:u-nw} taking $u(x)$ to $u_\nwarrow(x)$ solving \eqref{p4} for shifted parameters $(\Theta_0-\tfrac{1}{2},\Theta_\infty+\tfrac{1}{2})$ is determinate provided that $\Theta_0(\Theta_\infty-\Theta_0)\neq 0$.
\end{itemize}
\label{prop:Baecklund}
\end{proposition}
As is well-known, the four lines in the parameter space $\Theta_\infty\pm\Theta_0=0$ and $\Theta_\infty\pm\Theta_0=1$ give precisely the parameter values where the Painlev\'e-IV equation \eqref{p4} admits solutions in terms of classical special functions, namely those solving the linear second-order equation related in the usual way to the Riccati equation consistent with \eqref{p4}.

\subsection{Riemann-Hilbert representation of gO rationals}
\label{sec:RHPgO}
Now we carry out the program outlined at the beginning of Section~\ref{sec:rational-isomonodromy} to arrive at a Riemann-Hilbert representation of the rational solutions of the Painlev\'e-IV equation \eqref{p4} in the gO family.  The procedure begins with the selection of a seed solution for the family, which we take to correspond to the
special point $(\Theta_0,\Theta_\infty)=(\tfrac{1}{6},\tfrac{1}{2})\in\Lambda_\mathrm{gO}$.  The rational solution of \eqref{p4} may be obtained equivalently from any row of Table~\ref{tab:gO} for $m=n=0$, which gives simply $u(x)=-\tfrac{2}{3}x$. 

\subsubsection{Sowing the seed:  solving the direct monodromy problem and formulating the inverse monodromy problem}
\label{sec:gO-sowing}
The Lax pair equations \eqref{eq:PIV-Lax-Pair} for the seed solution involve $\Theta_0=\tfrac{1}{6}$, $\Theta_\infty=\tfrac{1}{2}$, $u(x)=-\tfrac{2}{3}x$, a nontrivial solution $y(x)$ of the first-order linear equation \eqref{eq:yODE} that we take without loss of generality to be $y(x)=\ee^{-\frac{2}{3}x^2}$, and $z(x)=\tfrac{1}{4}(-u'(x)+u(x)^2+2xu(x)+4\Theta_0)=\tfrac{1}{3}-\tfrac{2}{9}x^2$ (see \eqref{eq:LaxPair-z-define}).
%The $y$ function corresponding to $u(x)$ given by \eqref{eq:u-seed} satisfies
%\eq
%\frac{y_x}{y} = -u-2x = \frac{2}{3}x - 2x = -\frac{4}{3}x.
%\endeq
%This equation is separable with solution 
%\eq
%y=c\ee^{-\tfrac{2}{3}x^2}.
%\endeq
%Without loss of generality, take $c=1$.  Then using
%\eq
%z=\frac{1}{4}\left(-u_x+u^2+2xu+4\Theta_0\right)=\frac{1}{4}\left(\frac{2}{3}+\frac{4}{9}x^2-\frac{4}{3}x^2 + \frac{2}{3}\right)=\frac{1}{3}-\frac{2}{9}x^2
%\endeq
%which leads to
%\eq
%z-\Theta_0-\Theta_\infty=-\frac{1}{3}-\frac{2}{9}x^2,
%\endeq
%the $x$-equation in the Lax pair takes the form 
In particular, the $x$-equation $\mathbf{\Psi}_x=\mathbf{X\Psi}$ in the Lax pair \eqref{eq:PIV-Lax-Pair} takes the form
\eq
{\bf\Psi}_x=\bpm \lambda & \ee^{-\frac{2}{3}x^2}\\
(-\frac{2}{3}-\frac{4}{9}x^2)\ee^{\frac{2}{3}x^2} & -\lambda\epm{\bf\Psi}.
\endeq
The exponential factors in the coefficient matrix are easily removed with the help of a gauge transformation:  $\mathbf{\Psi}=\ee^{-\frac{1}{3}x^2\sigma_3}\mathbf{\Psi}^{[1]}$,
which 
%Making the change of variables
%\eq
%{\bf\Psi} = \bpm \exp(-\frac{1}{3}x^2) & 0 \\ 0 & \exp(\frac{1}{3}x^2) \epm {\bf\Phi}
%\label{eq:exponential-gauge}
%\endeq
leads to the equivalent system
\eq
{\bf\Psi}^{[1]}_x=
%\mathbf{\Psi}^{[1]},\quad
%\mathbf{U}^{[1]}(x,\lambda):=
\bpm \lambda+\frac{2}{3}x & 1\\
-\frac{2}{3}-\frac{4}{9}x^2 & -\lambda-\frac{2}{3}x\epm \mathbf{\Psi}^{[1]}.
\label{eq:Phi-x}
\endeq
%Writing ${\bf\Phi}=(\phi_1,\phi_2)^\mathsf{T}$, this equation is 
%\eq
%\phi_{1x} = \left(\lambda+\frac{2}{3}x\right)\phi_1 + \phi_2, \quad \phi_{2x} = 2\left(x-\frac{2}{3}\right)\phi_1 - \left(\lambda+\frac{2}{3}x\right)\phi_2.
%\endeq
%Differentiating the first equation, and then using the first equation to 
%substitute for $\phi_2$ and the second equation to substitute for $\phi_{2x}$ 
%gives 
%\eq
%\phi_{1xx} = \left[\left(\lambda+\frac{2}{3}x\right)^2+2x-\frac{2}{3}\right]\phi_1.
%\endeq
%\textcolor{blue}{Following this reasoning from my $\bf\Phi$ system I get instead
Using the first equation in this system to explicitly eliminate the second row yields a closed equation on elements $\Psi^{[1]}_{1k}$ of the first row:
\eq
\Psi^{[1]}_{1k,xx}=\left(\lambda^2+\tfrac{4}{3}\lambda x\right)\Psi^{[1]}_{1k},\quad k=1,2.
\endeq
This is easily transformed into Airy's equation.  Indeed, setting $w=(\tfrac{4}{3}\lambda)^{\frac{1}{3}}(x+\tfrac{3}{4}\lambda)$ for fixed $\lambda$ we arrive at
\eq
\Psi^{[1]}_{1k,ww}=w\Psi^{[1]}_{1k},\quad k=1,2.
\endeq
%}
%This equation can be solved for $\phi_1$ in terms of parabolic cylinder 
%functions, after which $\phi_2$ can be obtained from 
Once the first row is determined, the elements of the second row follow from the relation
\eq
\Psi^{[1]}_{2k} = \Psi^{[1]}_{1k,x} - \left(\lambda+\tfrac{2}{3}x\right)\Psi^{[1]}_{1k},\quad k=1,2.
\label{eq:phi2-phi1}
\endeq
Note that in terms of $\mathbf{\Psi}^{[1]}$, the $\lambda$-equation in the Lax pair \eqref{eq:PIV-Lax-Pair} for the gO seed solution takes the form
\eq
\mathbf{\Psi}^{[1]}_\lambda=
%\mathbf{A}^{[1]}(x,\lambda)\mathbf{\Psi}^{[1]},\quad
%\mathbf{A}^{[1]}(x,\lambda):=
\bpm
\lambda+x+\lambda^{-1}(\tfrac{2}{9}x^2-\tfrac{1}{6}) & 1+\tfrac{1}{3}\lambda^{-1}x\\
-\tfrac{4}{9}x^2-\tfrac{2}{3} +\lambda^{-1}x(\tfrac{2}{9}-\tfrac{4}{27}x^2) & -\lambda-x-\lambda^{-1}(\tfrac{2}{9}x^2-\tfrac{1}{6})\epm \mathbf{\Psi}^{[1]}.
%=
%\bpm
%\lambda + x + \lambda^{-1}(\Theta_0-z) & 1-\frac{1}{2}\lambda^{-1}u\\
%2\left(z-\Theta_0-\Theta_\infty+\frac{z}{\lambda u}(z-2\Theta_0)\right) & -\lambda-x-\lambda^{-1}(\Theta_0-z)
%\epm,
\label{eq:Phi-lambda}
\endeq
%where $u=-\tfrac{2}{3}x$ and $z=\tfrac{1}{3}-\tfrac{2}{9}x^2$, while $\Theta_0$ and $\Theta_\infty$ are given by \eqref{eq:Theta-ThetaInfty}.

The above calculations suggest the utility of the independent variables $w$ and $\mu:=\lambda$ in place of $(x,\lambda)$.  The differentiation formulas needed to effect the change of variables are
\eq
\frac{\partial}{\partial x}=\left(\tfrac{4}{3}\mu\right)^{\frac{1}{3}}\frac{\partial}{\partial w}\quad\text{and}\quad
\frac{\partial}{\partial\lambda}=\frac{\partial}{\partial\mu}+\left(\tfrac{1}{3}\frac{w}{\mu}+\tfrac{3}{4}\left(\tfrac{4}{3}\mu\right)^{\frac{1}{3}}\right)\frac{\partial}{\partial w}.
\label{eq:Jacobian}
\endeq
Next, it is convenient to introduce a subsequent gauge transformation in order to arrive at a $w$-equation for which the second row can be solved in terms of functions of $w$ alone (rather than also involving $\mu$).  Noting that the relation \eqref{eq:phi2-phi1} implies that also
\eq
\Psi^{[1]}_{2k}=\left(\tfrac{4}{3}\mu\right)^{\frac{1}{3}}\Psi^{[1]}_{1k,w}-\left(\tfrac{1}{2}\mu+\tfrac{2}{3}\left(\tfrac{4}{3}\mu\right)^{-\frac{1}{3}}w\right)\Psi^{[1]}_{1k}, \quad k=1,2,
\endeq
we introduce the ``shearing'' transformation
\eq
\mathbf{\Psi}^{[1]} = \mathbf{G}(w,\mu)\mathbf{\Psi}^{[2]},\quad
\mathbf{G}(w,\mu):=\bpm 1 & 0\\-\left(\tfrac{1}{2}\mu +\tfrac{2}{3}\left(\tfrac{4}{3}\mu\right)^{-\frac{1}{3}}w\right) & \left(\tfrac{4}{3}\mu\right)^{\frac{1}{3}}\epm.
\label{eq:shear}
\endeq
After some computation, it then follows from \eqref{eq:Phi-x}, \eqref{eq:Jacobian}, and \eqref{eq:shear} that
\eq
\mathbf{\Psi}^{[2]}_w = 
%\mathbf{U}^{(2)}(X,\Lambda)
\bpm 0&1\\w & 0\epm
\mathbf{\Psi}^{[2]}.
\label{eq:Gamma-X}
\endeq
%where
%\eq
%\mathbf{U}^{(2)}(X,\Lambda):=\left(\frac{4}{3}\Lambda\right)^{-1/3}\mathbf{G}(X,\Lambda)^{-1}\mathbf{U}^{(1)}(x,\lambda)\mathbf{G}(X,\Lambda)-\mathbf{G}(X,\Lambda)^{-1}\mathbf{G}_X(X,\Lambda) = \bpm 0&1\\X & 0\epm.
%\endeq
Similarly, combining this result with \eqref{eq:Phi-lambda}, \eqref{eq:Jacobian}, and \eqref{eq:shear},
\eq
\mathbf{\Psi}^{[2]}_\mu = 
%\mathbf{A}^{(2)}(X,\Lambda)
-\frac{1}{6\mu}
\mathbf{\Psi}^{[2]}.
\label{eq:Gamma-Lambda}
\endeq
%where
%\begin{multline}
%\mathbf{A}^{(2)}(X,\Lambda):=\mathbf{G}(X,\Lambda)^{-1}\mathbf{A}^{(1)}(x,\lambda)\mathbf{G}(X,\Lambda)\\
%{}-\left(\frac{X}{3\Lambda}+\frac{3}{4}\left(\frac{4}{3}\Lambda\right)^{1/3}\right)
%\left(\frac{4}{3}\Lambda\right)^{-1/3}\mathbf{G}(X,\Lambda)^{-1}\mathbf{U}^{(1)}(x,\lambda)\mathbf{G}(X,\Lambda) - \mathbf{G}(X,\Lambda)^{-1}\mathbf{G}_\Lambda(X,\Lambda) = -\frac{1}{6\Lambda}\mathbb{I}.
%\end{multline}
After this simplification, it is completely clear that every simultaneous fundamental solution matrix of 
\eqref{eq:Gamma-X} and \eqref{eq:Gamma-Lambda} has the form
\eq
\mathbf{\Psi}^{[2]}=\mu^{-\frac{1}{6}}\bpm f_1(w) & f_2(w)\\ f_1'(w) & f_2'(w)\epm \mathbf{C}
\endeq
where $\mathbf{C}$ is a matrix independent of both $w$ and $\mu$ with $\det(\mathbf{C})\neq 0$ and $f_j(w)$, $j=1,2$ form a fundamental pair of solutions of Airy's equation $f''(w)=wf(w)$.  Putting the pieces together we find the following result.
\begin{lemma}
Fix a fundamental pair $f_1(\cdot)$, $f_2(\cdot)$ of solutions of Airy's equation $f''(w)=wf(w)$, a simply connected domain $D\subset\mathbb{C}\setminus\{0\}$ and arbitrary branches of $\lambda^{\frac{4}{3}}$ and $\lambda^{-\frac{1}{6}}$ analytic on $D$.  For $\lambda\in D$ and $x\in\mathbb{C}$, define the matrix function
\eq
\mathbf{F}(\lambda,x):=
\lambda^{-\frac{1}{6}}\ee^{-\frac{1}{3}x^2\sigma_3}\bpm 1 & 0\\-\lambda-\tfrac{2}{3}x & (\tfrac{4}{3}\lambda)^{\frac{1}{3}}
\epm\bpm f_1((\tfrac{3}{4})^{\frac{2}{3}}\lambda^{\frac{4}{3}}(1+\tfrac{4}{3}\lambda^{-1}x)) & f_2((\tfrac{3}{4})^{\frac{2}{3}}\lambda^{\frac{4}{3}}(1+\tfrac{4}{3}\lambda^{-1}x))\\
f_1'((\tfrac{3}{4})^{\frac{2}{3}}\lambda^{\frac{4}{3}}(1+\tfrac{4}{3}\lambda^{-1}x)) & f_2'((\tfrac{3}{4})^{\frac{2}{3}}\lambda^{\frac{4}{3}}(1+\tfrac{4}{3}\lambda^{-1}x))\epm.
\endeq
Let $\Theta_0=\tfrac{1}{6}$, $\Theta_\infty=\tfrac{1}{2}$, and consider the exact solution $u(x)=-\tfrac{2}{3}x$ of the corresponding Painlev\'e-IV equation \eqref{p4}.  If $y(x)=\ee^{-\frac{2}{3}x^2}$, then the Lax pair equations \eqref{eq:PIV-Lax-Pair} are simultaneously solvable for all $(\lambda,x)\in D\times\mathbb{C}$, and every simultaneous solution matrix has the form
\eq
\mathbf{\Psi}(\lambda,x)=\mathbf{F}(\lambda,x)\mathbf{C},\quad (\lambda,x)\in D\times\mathbb{C},
\label{eq:general-solution}
\endeq
where $\mathbf{C}$ is a matrix independent of both $x$ and $\lambda$.
\label{lem:gO-direct-general}
\end{lemma}

\begin{remark}
In \cite[Section 3]{NovokshenovS14}, the authors obtain analogues of these results by a different method.  Namely, they observe that by fixing $x=0$, the equation $\mathbf{\Psi}_\lambda=\mathbf{\Lambda}\mathbf{\Psi}$ (cf.\@ \eqref{eq:PIV-Lax-Pair}--\eqref{eq:JM-A}), which is usually intractable from the point of view of classical special functions, reduces for the gO seed to a specific confluent hypergeometric equation solvable in terms of Whittaker functions.  For the specific parameters involved, the Whittaker functions reduce to Airy functions, see \cite[eqn.\@ 13.18.10]{DLMF}.  This results in a computation of the Stokes matrices that agrees with our calculations written in \eqref{eq:gO-Stokes} (and derived below) up to a constant diagonal conjugation.  No attempt is made in \cite{NovokshenovS14} to calculate any connection matrices.
\label{rem:NovokshenovS14}
\end{remark}

With this result in hand, we now consider how to choose the matrices $\mathbf{C}=\mathbf{C}^{(\infty)}_j$, $j=1,2,3,4$, corresponding to the four solutions $\mathbf{\Psi}=\mathbf{\Psi}^{(\infty)}_j(\lambda;x)$ to achieve the normalization condition \eqref{eq:infty-asymp} for each Stokes sector abutting the irregular singular point at $\lambda=\infty$.
To make this calculation precise, we interpret all fractional powers of $\lambda$ appearing in \eqref{eq:general-solution} as principal branches:  $\lambda^p=\mathrm{e}^{p\log(\lambda)}$, $-\pi<\mathrm{Im}(\log(\lambda))<\pi$.  This means in particular that $\lambda^p$ has in general a different meaning on the common boundary of $S_3$ and $S_4$, depending on which of those two sectors is under consideration.  
For each sector $S_j$ in turn, we shall choose first for $f_1$ and $f_2$ a specific fundamental pair of solutions of $f''(w)=wf(w)$ that exhibits no Stokes phenomenon as $\lambda\to\infty$ in the sector.  Then we use well-known asymptotic formul\ae\ for Airy functions of large argument to determine the corresponding matrix $\mathbf{C}^{(\infty)}_j$.  Once $\mathbf{\Psi}_j^{(\infty)}(\lambda,x)$ has been determined for $j=1,\dots,4$, we will build a solution $\mathbf{\Psi}^{(0)}(\lambda,x)$ that satisfies the condition \eqref{eq:zero-asymp}.

\subsubsection*{The solution $\mathbf{\Psi}^{(\infty)}_1(\lambda,x)$}  
If $\lambda$ is large in the sector $S_1$, then $w=(\tfrac{3}{4})^{\frac{2}{3}}\lambda^{\frac{4}{3}}(1+\tfrac{4}{3}\lambda^{-1}x)$ is also large, and $-\tfrac{1}{2}\pi\le\arg(\lambda)\le 0$ implies that $-\frac{2}{3}\pi-\epsilon\le\arg(w)\le \epsilon$ holds for every $\epsilon>0$ if $\lambda$ is large enough given $\epsilon$.   The solutions $f_1(w)=\mathrm{Ai}(w)$ and $f_2(w)=\mathrm{Ai}(\ee^{\frac{2\ii\pi}{3}}w)$ do not exhibit Stokes phenomenon in this sector as $w\to\infty$ for $\epsilon$ small enough.  Using \cite[Eqns. 9.7.5 \& 9.7.6]{DLMF} and composing with the definition of $w$ gives
\eq
f_1(w)=\frac{\ee^{-\frac{1}{3}x^2}}{2\sqrt{\pi}}\left(\tfrac{4}{3}\right)^{\frac{1}{6}}\lambda^{-\frac{1}{3}}\ee^{-(\frac{1}{2}\lambda^2+x\lambda)}(1+(\tfrac{2}{27}x^3-\tfrac{1}{3}x)\lambda^{-1}+\bo(\lambda^{-2})),\quad\lambda\to\infty,\quad\lambda\in S_1,
\label{eq:f1-S1}
\endeq
\eq
f_1'(w)=-\frac{\ee^{-\frac{1}{3}x^2}}{2\sqrt{\pi}}\left(\tfrac{3}{4}\right)^{\frac{1}{6}}\lambda^{\frac{1}{3}}\ee^{-(\frac{1}{2}\lambda^2+x\lambda)}(1+(\tfrac{2}{27}x^3+\tfrac{1}{3}x)\lambda^{-1}+\bo(\lambda^{-2})),\quad\lambda\to\infty,\quad\lambda\in S_1,
\label{eq:f1prime-S1}
\endeq
\eq
f_2(w)=\ee^{-\frac{\ii\pi}{6}}\frac{\ee^{\frac{1}{3}x^2}}{2\sqrt{\pi}}\left(\tfrac{4}{3}\right)^{\frac{1}{6}}\lambda^{-\frac{1}{3}}\ee^{\frac{1}{2}\lambda^2+x\lambda}(1-(\tfrac{2}{27}x^3+\tfrac{1}{3}x)\lambda^{-1}+\bo(\lambda^{-2})),\quad\lambda\to\infty,\quad\lambda\in S_1,
\endeq
and
\eq
f_2'(w)=\ee^{-\frac{\ii\pi}{6}}\frac{\ee^{\frac{1}{3}x^2}}{2\sqrt{\pi}}\left(\tfrac{3}{4}\right)^{\frac{1}{6}}\lambda^{\frac{1}{3}}\ee^{\frac{1}{2}\lambda^2+x\lambda}(1-(\tfrac{2}{27}x^3-\tfrac{1}{3}x)\lambda^{-1}+\bo(\lambda^{-2})),\quad\lambda\to\infty,\quad\lambda\in S_1.
\endeq
A straightforward computation then shows that if $\mathbf{\Psi}_1^{(\infty)}(\lambda,x)$ has the form \eqref{eq:general-solution} with the above choices for $f_1(w)$ and $f_2(w)$, and with the constant matrix $\mathbf{C}=\mathbf{C}_1^{(\infty)}$, then using $\Theta_\infty=\tfrac{1}{2}$,
\begin{multline}
\mathbf{\Psi}_1^{(\infty)}(\lambda,x)\lambda^{\Theta_\infty\sigma_3}\ee^{-(\frac{1}{2}\lambda^2+x\lambda)\sigma_3}=
\bpm
\bo(\lambda^{-1}) & \displaystyle\tfrac{1}{2\sqrt{\pi}}\left(\tfrac{4}{3}\right)^{\frac{1}{6}}\ee^{-\frac{\ii\pi}{6}}+\bo(\lambda^{-1})\\
\displaystyle -\tfrac{1}{\sqrt{\pi}}\left(\tfrac{4}{3}\right)^{\frac{1}{6}} + \bo(\lambda^{-1}) & \bo(\lambda^{-1})
\epm \\
{}\cdot\lambda^{\frac{1}{2}\sigma_3}\ee^{-(\frac{1}{2}\lambda^2+x\lambda)\sigma_3}\mathbf{C}^{(\infty)}_1\ee^{-(\frac{1}{2}\lambda^2+x\lambda)\sigma_3}\lambda^{\frac{1}{2}\sigma_3},\quad\lambda\to\infty,\quad\lambda\in S_1.
\end{multline}
Therefore to achieve the desired asymptotic normalization condition \eqref{eq:infty-asymp} for $j=1$, we must take
\eq
\mathbf{C}_1^{(\infty)}:=\bpm 0 & \displaystyle -\sqrt{\pi}\left(\tfrac{3}{4}\right)^{\frac{1}{6}}\\
\displaystyle 2\sqrt{\pi}\left(\tfrac{3}{4}\right)^{\frac{1}{6}}\ee^{\frac{\ii\pi}{6}} & 0\epm.
\endeq
This completes the determination of the normalized simultaneous fundamental solution matrix $\mathbf{\Psi}_1^{(\infty)}(\lambda,x)$.

\subsubsection*{The solution $\mathbf{\Psi}^{(\infty)}_2(\lambda,x)$}  
When $\lambda\in S_2$, we have $-\epsilon\le\arg(w)\le \frac{2}{3}\pi +\epsilon$ as $\lambda\to\infty$ and hence also $w\to\infty$.
In this case, to avoid Stokes phenomenon we choose solutions $f_1(w)=\mathrm{Ai}(w)$ and $f_2(w)=\mathrm{Ai}(\ee^{-\frac{2\ii\pi}{3}}w)$.
Thus, again composing the definition of $w$ with \cite[Eqns. 9.7.5 \& 9.7.6]{DLMF}, the expansions \eqref{eq:f1-S1}--\eqref{eq:f1prime-S1} are also valid as $\lambda\to\infty$ for $\lambda\in S_2$ and we have
%\eq
%f_1(w)=\frac{\ee^{-\frac{1}{3}x^2}}{2\sqrt{\pi}}\left(\tfrac{4}{3}\right)^{\frac{1}{6}}\lambda^{-\frac{1}{3}}\ee^{-(\frac{1}{2}\lambda^2+x\lambda)}(1+(\tfrac{2}{27}x^3-\tfrac{1}{3}x)\lambda^{-1}+O(\lambda^{-2})),\quad\lambda\to\infty,\quad\lambda\in S_2.
%\endeq
%\eq
%f_1'(w)=-\frac{\ee^{-\frac{1}{3}x^2}}{2\sqrt{\pi}}\left(\tfrac{3}{4}\right)^{\frac{1}{6}}\lambda^{\frac{1}{3}}\ee^{-(\frac{1}{2}\lambda^2+x\lambda)}(1+(\tfrac{2}{27}x^3+\tfrac{1}{3}x)\lambda^{-1}+O(\lambda^{-2})),\quad\lambda\to\infty,\quad\lambda\in S_2.
%\endeq
\eq
f_2(w)=\ee^{\frac{\ii\pi}{6}}\frac{\ee^{\frac{1}{3}x^2}}{2\sqrt{\pi}}\left(\tfrac{4}{3}\right)^{\frac{1}{6}}\lambda^{-\frac{1}{3}}\ee^{\frac{1}{2}\lambda^2+x\lambda}(1-(\tfrac{2}{27}x^3+\tfrac{1}{3}x)\lambda^{-1}+\bo(\lambda^{-2})),\quad\lambda\to\infty,\quad\lambda\in S_2
\label{eq:f2-S2}
\endeq
and
\eq
f_2'(w)=\ee^{\frac{\ii\pi}{6}}\frac{\ee^{\frac{1}{3}x^2}}{2\sqrt{\pi}}\left(\tfrac{3}{4}\right)^{\frac{1}{6}}\lambda^{\frac{1}{3}}\ee^{\frac{1}{2}\lambda^2+x\lambda}(1-(\tfrac{2}{27}x^3-\tfrac{1}{3}x)\lambda^{-1}+\bo(\lambda^{-2})),\quad\lambda\to\infty,\quad\lambda\in S_2.
\label{eq:f2prime-S2}
\endeq
Taking these choices for $f_1(w)$ and $f_2(w)$ in \eqref{eq:general-solution} with $\mathbf{\Psi}=\mathbf{\Psi}_2^{(\infty)}(\lambda,x)$ and $\mathbf{C}=\mathbf{C}^{(\infty)}_2$, for $\Theta_\infty=\tfrac{1}{2}$,
\begin{multline}
\mathbf{\Psi}_2^{(\infty)}(\lambda,x)\lambda^{\Theta_\infty\sigma_3}\ee^{-(\frac{1}{2}\lambda^2+x\lambda)\sigma_3}=
\bpm \bo(\lambda^{-1}) & \displaystyle\tfrac{1}{2\sqrt{\pi}}\left(\tfrac{4}{3}\right)^{\frac{1}{6}}\ee^{\frac{\ii\pi}{6}} + \bo(\lambda^{-1})\\
\displaystyle -\tfrac{1}{\sqrt{\pi}}\left(\tfrac{4}{3}\right)^{\frac{1}{6}}+\bo(\lambda^{-1}) & \bo(\lambda^{-1})\epm\\
{}\cdot\lambda^{\frac{1}{2}\sigma_3}\ee^{-(\frac{1}{2}\lambda^2+x\lambda)\sigma_3}\mathbf{C}^{(\infty)}_2
\ee^{-(\frac{1}{2}\lambda^2+x\lambda)\sigma_3}\lambda^{\frac{1}{2}\sigma_3},\quad\lambda\to\infty,\quad\lambda\in S_2.
\end{multline}
Therefore, to achieve the desired normalization condition \eqref{eq:infty-asymp} we must take
\eq
\mathbf{C}^{(\infty)}_2:=\bpm 0 & \displaystyle -\sqrt{\pi}\left(\tfrac{3}{4}\right)^{\frac{1}{6}}\\\displaystyle 2\sqrt{\pi}\left(\tfrac{3}{4}\right)^{\frac{1}{6}}\ee^{-\frac{\ii\pi}{6}} & 0\epm,
\endeq
which completes the construction of $\mathbf{\Psi}^{(\infty)}_2(\lambda,x)$.  We observe here that $\mathbf{C}^{(\infty)}_2=\mathbf{C}_1^{(\infty)*}$ (element-wise complex conjugation), which is consistent with the fact that for $\mathbf{\Psi}_2^{(\infty)}$ we selected a basis $(f_1(w),f_2(w))$ whose elements are the Schwarz reflections of the basis elements selected to construct $\mathbf{\Psi}_1^{(\infty)}$.

\subsubsection*{The solution $\mathbf{\Psi}^{(\infty)}_3(\lambda,x)$}  
The sector $S_3$ corresponds to $\tfrac{2}{3}\pi-\epsilon\le\arg(w)\le\tfrac{4}{3}\pi +\epsilon$ as $\lambda\to\infty$.  We hence choose the basis $f_1(w):=\mathrm{Ai}(\ee^{-\frac{2\ii\pi}{3}}w)$ and $f_2(w):=\mathrm{Ai}(\ee^{-\frac{4\ii\pi}{3}}w)=\mathrm{Ai}(\ee^{\frac{2\ii\pi}{3}}w)$ to avoid Stokes phenomenon.  It follows that the expansions \eqref{eq:f2-S2}--\eqref{eq:f2prime-S2} are also valid (for $f_1(w)$ and $f_1'(w)$ in place of $f_2(w)$ and $f_2'(w)$ on the left-hand sides) as $\lambda\to\infty$ for $\lambda\in S_3$, and that
%\eq
%f_1(w)=\ee^{\frac{\ii\pi}{6}}\frac{\ee^{\frac{1}{3}x^2}}{2\sqrt{\pi}}(\tfrac{4}{3})^{\frac{1}{6}}\lambda^{-\frac{1}{3}}\ee^{\frac{1}{2}\lambda^2+x\lambda}(1-(\tfrac{2}{27}x^3+\tfrac{1}{3}x)\lambda^{-1}+O(\lambda^{-2})),\quad\lambda\to\infty,\quad\lambda\in S_3.
%\endeq
%\eq
%f_1'(w)=\ee^{\frac{\ii\pi}{6}}\frac{\ee^{\frac{1}{3}x^2}}{2\sqrt{\pi}}(\tfrac{3}{4})^{\frac{1}{6}}\lambda^{\frac{1}{3}}\ee^{\frac{1}{2}\lambda^2+x\lambda}(1-(\tfrac{2}{27}x^3-\tfrac{1}{3}x)\lambda^{-1}+O(\lambda^{-2})),\quad\lambda\to\infty,\quad\lambda\in S_3.
%\endeq
\eq
f_2(w)=\ee^{\frac{\ii\pi}{3}}\frac{\ee^{-\frac{1}{3}x^2}}{2\sqrt{\pi}}(\tfrac{4}{3})^{\frac{1}{6}}\lambda^{-\frac{1}{3}}\ee^{-(\frac{1}{2}\lambda^2+x\lambda)}(1+(\tfrac{2}{27}x^3-\tfrac{1}{3}x)\lambda^{-1}+\bo(\lambda^{-2})),\quad\lambda\to\infty,\quad\lambda\in S_3
\endeq
and
\eq
f_2'(w)=-\ee^{\frac{\ii\pi}{3}}\frac{\ee^{-\frac{1}{3}x^2}}{2\sqrt{\pi}}(\tfrac{3}{4})^{\frac{1}{6}}\lambda^{\frac{1}{3}}\ee^{-(\frac{1}{2}\lambda^2+x\lambda)}(1+(\tfrac{2}{27}x^3+\tfrac{1}{3}x)\lambda^{-1}+\bo(\lambda^{-2})),\quad\lambda\to\infty,\quad\lambda\in S_3.
\endeq
With these choices for $f_1(w)$ and $f_2(w)$ in \eqref{eq:general-solution} with $\mathbf{\Psi}=\mathbf{\Psi}_3^{(\infty)}(\lambda,x)$ and $\mathbf{C}=\mathbf{C}^{(\infty)}_3$, we find that
\begin{multline}
\mathbf{\Psi}_3^{(\infty)}(\lambda,x)\lambda^{\Theta_\infty\sigma_3}\ee^{-(\frac{1}{2}\lambda^2+x\lambda)\sigma_3}=\bpm
\displaystyle
\tfrac{1}{2\sqrt{\pi}}\left(\tfrac{4}{3}\right)^{\frac{1}{6}}\ee^{\frac{\ii\pi}{6}}+\bo(\lambda^{-1}) & \bo(\lambda^{-1})\\
\bo(\lambda^{-1}) & \displaystyle \tfrac{1}{\sqrt{\pi}}\left(\tfrac{4}{3}\right)^{\frac{1}{6}}\ee^{-\frac{2\ii\pi}{3}}+\bo(\lambda^{-1})\epm\\
{}\cdot\lambda^{-\frac{1}{2}\sigma_3}\ee^{(\frac{1}{2}\lambda^2+x\lambda)\sigma_3}\mathbf{C}^{(\infty)}_3
\ee^{-(\frac{1}{2}\lambda^2+x\lambda)\sigma_3}\lambda^{\frac{1}{2}\sigma_3},\quad\lambda\to\infty,\quad\lambda\in S_3.
\end{multline}
Therefore, to achieve the desired normalization condition \eqref{eq:infty-asymp} for $j=3$ we must take
\eq
\mathbf{C}^{(\infty)}_3:=\bpm \displaystyle 2\sqrt{\pi}\left(\tfrac{3}{4}\right)^{\frac{1}{6}}\ee^{-\frac{\ii\pi}{6}} & 0\\0 & \displaystyle \sqrt{\pi}\left(\tfrac{3}{4}\right)^{\frac{1}{6}}\ee^{\frac{2\ii\pi}{3}}\epm.
\endeq
This completes the construction of $\mathbf{\Psi}_3^{(\infty)}(\lambda,x)$.

\subsubsection*{The solution $\mathbf{\Psi}^{(\infty)}_4(\lambda,x)$}  
The sector $S_4$ corresponds to $-\tfrac{4}{3}\pi-\epsilon\le\arg(w)\le-\tfrac{2}{3}\pi +\epsilon$ as $\lambda\to\infty$.  We choose the basis $f_1(w):=\mathrm{Ai}(\ee^{\frac{2\ii\pi}{3}}w)$ and $f_2(w):=\mathrm{Ai}(\ee^{\frac{4\ii\pi}{3}}w)=\mathrm{Ai}(\ee^{-\frac{2\ii\pi}{3}}w)$ to avoid Stokes phenomenon.  Observing that these basis functions are the Schwarz reflections of those selected to construct $\mathbf{\Psi}_3^{(\infty)}$, one can check that taking $\mathbf{\Psi}=\mathbf{\Psi}_4^{(\infty)}(\lambda,x)$ and $\mathbf{C}=\mathbf{C}^{(\infty)}_4$ in \eqref{eq:general-solution}, the normalization condition \eqref{eq:infty-asymp} holds provided that
\eq
\mathbf{C}^{(\infty)}_4:=\mathbf{C}_3^{(\infty)*} = \bpm \displaystyle 2\sqrt{\pi}\left(\tfrac{3}{4}\right)^{\frac{1}{6}}\ee^{\frac{\ii\pi}{6}} & 0\\0 & \displaystyle \sqrt{\pi}\left(\tfrac{3}{4}\right)^{\frac{1}{6}}\ee^{-\frac{2\ii\pi}{3}}\epm.
\endeq
This completes the construction of $\mathbf{\Psi}^{(\infty)}_4(\lambda,x)$.

\subsubsection*{The solution $\mathbf{\Psi}^{(0)}(\lambda,x)$.}  
Finally, we determine the constant matrix $\mathbf{C}=\mathbf{C}^{(0)}$ in \eqref{eq:general-solution} so that with $\mathbf{\Psi}=\mathbf{\Psi}^{(0)}(\lambda,x)$ and $\Theta_0=\tfrac{1}{6}$ the condition \eqref{eq:zero-asymp} holds.
This is possible because the Fuchsian singularity at $\lambda=0$ of $\mathbf{\Psi}_\lambda=\mathbf{\Lambda\Psi}$ is nonresonant for $\Theta_0=\tfrac{1}{6}$.  For this calculation we may choose any basis of solutions of Airy's equation, so we select the same basis as in the definition of $\mathbf{\Psi}^{(\infty)}_1(\lambda,x)$, namely $f_1(w)=\mathrm{Ai}(w)$ and $f_2(w)=\mathrm{Ai}(\ee^{\frac{2\ii\pi}{3}}w)$.  This will make it easiest to relate $\mathbf{\Psi}^{(0)}(\lambda,x)$ explicitly to $\mathbf{\Psi}^{(\infty)}_1(\lambda,x)$.
We begin with the following elementary observation:  
%using \cite[Eqns.\@ 9.2.3--9.2.4]{DLMF}, 
it can be shown that with the above definitions, 
\eq
\tilde{f}_1(w):=f_1(w)-f_2(w)
\endeq
is a nontrivial solution of Airy's equation whose Taylor expansion at the origin contains only terms proportional to $w^{3n+1}$, $n=0,1,2,\dots$.  Likewise
\eq
\tilde{f}_2(w):=f_1(w)-\ee^{-\frac{2\ii\pi}{3}}f_2(w)
\endeq
is a nontrivial solution of Airy's equation whose Taylor expansion at the origin contains only terms proportional to $w^{3n}$, $n=0,1,2,\dots$.  Composing with $w=(\tfrac{4}{3}\lambda)^{\frac{1}{3}}(x+\tfrac{3}{4}\lambda)$ shows that $\lambda^{-\frac{1}{3}}\tilde{f}_1(w)$ and $\tilde{f}_2(w)$ are both analytic functions of $\lambda$ at $\lambda=0$.  Taking into account that $\Theta_0=\tfrac{1}{6}$ then shows that $\mathbf{\Psi}^{(0)}(\lambda,x)\lambda^{-\Theta_0\sigma_3}$ will be analytic at $\lambda=0$ if we take $\mathbf{\Psi}=\mathbf{\Psi}^{(0)}(\lambda,x)$ in the form \eqref{eq:general-solution} with the above choice of basis $f_1(w)$ and $f_2(w)$ and insist that $\mathbf{C}=\mathbf{C}^{(0)}$ has the form
\eq
\mathbf{C}^{(0)}:=\bpm 1 & 1\\ -1 & \ee^{\frac{\ii\pi}{3}}\epm\mathbf{D}
\endeq
where $\mathbf{D}$ is any constant invertible diagonal matrix.  Modulo the choice of $\mathbf{D}$, which we will make concrete below (see \eqref{eq:D-def}), this completes the construction of $\mathbf{\Psi}^{(0)}(\lambda,x)$.

The last step of sowing the seed is to formulate the inverse monodromy problem, which simply amounts to the calculation of the constant matrices relating the five simultaneous fundamental solution matrices of the Lax pair \eqref{eq:PIV-Lax-Pair}.  For the Stokes matrices $\mathbf{V}_{2,1}$, $\mathbf{V}_{2,3}$, $\mathbf{V}_{4,3}$, and $\mathbf{V}_{4,1}$ defined by \eqref{eq:V21-def}--\eqref{eq:V41-def}, we use the identity $\mathrm{Ai}(w) + \ee^{\frac{2\ii\pi}{3}}\mathrm{Ai}(\ee^{\frac{2\ii\pi}{3}}w) + \ee^{-\frac{2\ii\pi}{3}}\mathrm{Ai}(\ee^{-\frac{2\ii\pi}{3}}w)=0$ (see \cite[Eqn.\@ 9.2.12]{DLMF}) to explicitly relate the simultaneous solutions $\mathbf{\Psi}_j^{(\infty)}(\lambda,x)$, $j=1,\dots,4$, leading to \eqref{eq:gO-Stokes} in which 
$\Theta_\infty=\tfrac{1}{2}$.
The computation of the Stokes matrix $\mathbf{V}_{4,3}$ requires more care than the others, because one must take into account the jump in the principal branch of $\lambda^p$ across the negative real axis.  
For the connection matrices defined by \eqref{eq:V1V3-def}--\eqref{eq:V2V4-def}, first note that
since the basis of Airy functions $f_1(w)$ and $f_2(w)$ in the formula \eqref{eq:general-solution} is exactly the same for $\mathbf{\Psi}=\mathbf{\Psi}_1^{(\infty)}(\lambda,x)$ and $\mathbf{\Psi}=\mathbf{\Psi}^{(0)}(\lambda,x)$, \eqref{eq:V1V3-def} gives $\mathbf{V}_1=\mathbf{C}_1^{(\infty)-1}\mathbf{C}^{(0)}$.
%\eq
%\mathbf{\Psi}^{(0)}(\lambda,x)=\mathbf{\Psi}_1^{(\infty)}(\lambda,x)\mathbf{C}_1^{-1}\mathbf{C}^{(0)} 
%\endeq
Therefore, choosing without loss of generality that 
\eq
\mathbf{D}:=\bpm -\tfrac{2}{\sqrt{3}}(\tfrac{3}{4})^{\frac{1}{6}}\sqrt{\pi} & 0\\0 & -(\tfrac{3}{4})^{\frac{1}{6}}\sqrt{\pi}
\epm\ee^{\frac{\ii\pi}{6}\sigma_3} = -(\tfrac{4}{3})^{\tfrac{1}{12}}\sqrt{\pi}((\tfrac{4}{3})^{\frac{1}{4}}\ee^{\frac{\ii\pi}{6}})^{\sigma_3},
\label{eq:D-def}
\endeq
we obtain the unimodular connection matrix $\mathbf{V}_1$ from \eqref{eq:V1V3-def}, after which it is easiest to obtain $\mathbf{V}_2$, $\mathbf{V}_3$, and $\mathbf{V}_4$ by combining the Stokes matrices \eqref{eq:gO-Stokes} with the first three identities in \eqref{eq:consistency}.  In this way, we obtain the connection matrices given in \eqref{eq:gO-connection}.
From these formul\ae\ one can observe that $\mathbf{V}_1^*=\mathbf{V}_2^{-1}$ and $\mathbf{V}_3^*=\mathbf{V}_4^{-1}$, a useful symmetry in light of Corollary~\ref{cor:Schwarz} that explains our choice of the diagonal constant matrix $\mathbf{D}$ in \eqref{eq:D-def}.
%Note also that combining \eqref{eq:Psi4-Psi3}, \eqref{eq:Psi0-Psi3}, and \eqref{eq:Psi0-Psi4} shows that across that branch cut the following jump condition holds:
%\eq
%\lim_{\epsilon\downarrow 0}\mathbf{\Psi}^{(0)}(\lambda+\ii\epsilon,x)=
%\lim_{\epsilon\downarrow 0}\mathbf{\Psi}^{(0)}(\lambda-\ii\epsilon,x)\ee^{\ii\pi\sigma_3/3},\quad \lambda<0.
%\endeq

\subsubsection{Reaping the harvest:  use of Schlesinger transformations to span the gO parameter lattice}
The above arguments have shown that Riemann-Hilbert Problem~\ref{rhp:general} obviously has a solution when $(\Theta_0,\Theta_\infty)=(\tfrac{1}{6},\tfrac{1}{2})$ and the piecewise constant matrix $\mathbf{V}$ is defined on the jump contour $\Sigma$ in terms of the Stokes matrices \eqref{eq:gO-Stokes} and the connection matrices \eqref{eq:gO-connection}, namely the solution given by the formula \eqref{eq:Y-Psi}.  One can check that the function $u(x)$ returned from this solution via the formula \eqref{eq:yzu-define} is well-defined and coincides with the seed $u(x)=-\tfrac{2}{3}x$ from which we began.  Noting that the gO rational solution parameter lattice $\Lambda_\mathrm{gO}$ may be written as the set of points $(\Theta_0,\Theta_\infty)$ of the form $(\Theta_0,\Theta_\infty)=(\tfrac{1}{6},\tfrac{1}{2})+\mathbb{Z}(\tfrac{1}{2},\tfrac{1}{2}) + \mathbb{Z}(\tfrac{1}{2},-\tfrac{1}{2})$, none of which satisfy any of the conditions $\Theta_0=0$, $\Theta_\infty\pm\Theta_0=0$, or $\Theta_\infty\pm\Theta_0=1$, arbitrary iterations of the four Schlesinger transformations developed in Section~\ref{sec:sub-Schlesinger} and their coincident B\"acklund transformations from Section~\ref{sec:sub-Baecklund} can be applied to the seed to reach any point of $\Lambda_\mathrm{gO}$.  Since the only effect on the jump conditions of these Schlesinger transformations is to change the sign of $\mathbf{V}_{4,3}$ with each iteration, and since the B\"acklund transformations obviously map rational solutions to rational solutions, which are necessarily unique for given parameter values, we have arrived at the Riemann-Hilbert representation of the gO rational solutions of Painlev\'e-IV given in Theorem~\ref{thm:gO-RHP}, the proof of which we now complete.

\begin{proof}[Proof of Theorem~\ref{thm:gO-RHP}]
If $(\Theta_0,\Theta_\infty)=(\tfrac{1}{6},\tfrac{1}{2})$, then the statement is true with $u(x)=-\tfrac{2}{3}x$ and $u_\tw(x)=x^{-1}-\tfrac{2}{3}x$, the latter solving \eqref{p4} for parameters $(\Theta_{0,\tw},\Theta_{\infty,\tw})=(-\tfrac{1}{3},1)$.  By Proposition~\ref{prop:nonzero} we can apply the four Schlesinger transformations $\mathbf{Y}(\lambda;x)\mapsto\mathbf{Y}_\nearrow(\lambda;x)$, $\mathbf{Y}(\lambda;x)\mapsto\mathbf{Y}_{\swarrow}(\lambda;x)$, $\mathbf{Y}(\lambda;x)\mapsto\mathbf{Y}_\searrow(\lambda;x)$, and/or $\mathbf{Y}(\lambda;x)\mapsto\mathbf{Y}_\nwarrow(\lambda;x)$
to increment and/or decrement $\Theta_0$ and $\Theta_\infty$ by half-integers iteratively to reach any point in $\Lambda_\mathrm{gO}$ (any path from $(\tfrac{1}{6},\tfrac{1}{2})$ to $(\Theta_0,\Theta_\infty)\in\Lambda_\mathrm{gO}$ through $\Lambda_\mathrm{gO}$ suffices, because each path produces a solution of the same Riemann-Hilbert problem, which must be unique by Proposition~\ref{prop:RHP-basic}).  Each step in the lattice introduces a sign change in $\mathbf{V}_{4,3}$, which is automatically taken into account in the definition (see \eqref{eq:gO-Stokes}).  Next, we observe that the corresponding B\"acklund transformations 
$u(x)\mapsto u_\nearrow(x)$, $u(x)\mapsto u_\swarrow(x)$, $u(x)\mapsto u_\searrow(x)$, and $u(x)\mapsto u_\nwarrow(x)$
induced by the Schlesinger transformations all map rational functions to rational functions.
By uniqueness of rational solutions for Painlev\'e-IV it follows that $u(x)$ extracted from Riemann-Hilbert Problem~\ref{rhp:general} is the unique rational solution of \eqref{p4} with arbitrary parameters $(\Theta_0,\Theta_\infty)\in\Lambda_\mathrm{gO}$.  Since $u_\tw(x)$ is related to the rational function $u(x)$ by the B\"acklund transformation in \eqref{eq:Baecklund-3-to-1}, it is well-defined (since $u(x)$ cannot vanish identically as $\Theta_0\neq 0$ in the gO parameter lattice) and rational.  Hence it is the unique rational solution of \eqref{p4} for parameters $(\Theta_{0,\tw},\Theta_{\infty,\tw})$ defined by \eqref{eq:Baecklund-3-to-1} (these also lie in $\Lambda_\mathrm{gO}$).
%given by \eqref{eq:Thetas-m-n} that behaves as $u^{(m,n)}(x)=-\tfrac{2}{3}x+O(1)$ as $x\to\infty$, which proves the first statement for general $(m,n)\in\mathbb{Z}^2$.  It is known that for $m,n\in\mathbb{Z}_{\ge 0}$, \eqref{eq:generalized-Okamoto-rationals} defines a rational solution of \eqref{p4} for parameters given by \eqref{eq:Thetas-m-n}, and this solution obviously has the asymptotic behavior $u^{(I)}_{m,n}(x)\sim -\tfrac{2}{3}x$ as $x\to\infty$, so   
%Proposition~\ref{prop:uniqueness} proves the second statement.
%it therefore suffices to show that Riemann-Hilbert Problem~\ref{rhp:00} produces a rational solution via \eqref{eq:yzu-define} for the same parameters having the same leading asymptotic behavior as $x\to\infty$.  But f
\end{proof}

\subsection{Riemann-Hilbert representation of gH rationals}
\label{sec:RHPgH}
It is straightforward to check that if $\Theta_0=\Theta_\infty=\tfrac{1}{2}$ in \eqref{p4}, then $u(x)=-2x$ is an exact solution; this corresponds to taking $m=n=0$ in the type-$3$ row of Table~\ref{tab:gH}.  We shall use it as a seed in the same way that $u(x)=-\tfrac{2}{3}x$ was used in Section~\ref{sec:RHPgO} to derive a Riemann-Hilbert representation of the generalized Hermite rational solutions of Painlev\'e-IV.  

\subsubsection{Sowing the seed:  solving the direct monodromy problem and formulating the inverse monodromy problem}
When $(\Theta_0,\Theta_\infty,u(x))=(\tfrac{1}{2},\tfrac{1}{2},-2x)$, the differential equation \eqref{eq:yODE} has the general solution $y(x)=y_0$, a constant (assumed nonzero), and the quantity $z(x)$ defined in \eqref{eq:LaxPair-z-define} is $z(x)\equiv 1$.  Without loss of generality, we take $y_0=2$.  Therefore the Lax pair equations \eqref{eq:PIV-Lax-Pair} take a particularly simple form in this case, because the coefficient matrices defined generally in \eqref{eq:JM-A}--\eqref{eq:JM-U} are now upper triangular:
\eq
(\Theta_0,\Theta_\infty,u(x))=(\tfrac{1}{2},\tfrac{1}{2},-2x)\quad\implies\quad
\mathbf{\Lambda}=\bpm\lambda+x-\tfrac{1}{2}\lambda^{-1} & 2+2x\lambda^{-1}\\0 & -\lambda-x+\tfrac{1}{2}\lambda^{-1}\epm,\quad
\mathbf{X}=\bpm\lambda & 2\\0 & -\lambda\epm.
\endeq
Therefore, the second row elements of a simultaneous matrix solution $\mathbf{\Psi}$
satisfy a compatible first-order scalar system, whose general solution is easily seen to be
$\Psi_{2j}=c\ee^{-(\frac{1}{2}\lambda^2+x\lambda)}\lambda^{\frac{1}{2}}$, where $c$ is an integration constant (independent of both $x$ and $\lambda$).  Using this result, the first row elements $\Psi_{1j}$ then satisfy their own compatible scalar system, which can be solved with the help of an integrating factor proportional to $\Psi_{2j}$.  The result of these completely elementary calculations is the following.
\begin{lemma}
Fix a simply connected domain $D\subset\mathbb{C}\setminus\{0\}$ and a branch of $\lambda^\frac{1}{2}$ analytic on $D$.  Let $\Theta_0=\Theta_\infty=\tfrac{1}{2}$, and consider the exact solution $u(x)=-2x$ of the corresponding Painlev\'e-IV equation \eqref{p4}.  If $y(x)=2$, then the Lax pair equations \eqref{eq:PIV-Lax-Pair} are simultaneously solvable for all $(\lambda,x)\in D\times\mathbb{C}$, and every 
simultaneous solution matrix has the form
%\eq
%\bpm -\tfrac{1}{2}c_1y_0\lambda^{-1/2}\ee^{-\lambda^2/2-x\lambda} + d_1\lambda^{-1/2}\ee^{\lambda^2/2+x\lambda} & -\tfrac{1}{2}c_2y_0\lambda^{-1/2}\ee^{-\lambda^2-x\lambda} + d_2\lambda^{-1/2}\ee^{\lambda^2/2+x\lambda}\\
%c_1\lambda^{1/2}\ee^{-\lambda^2/2-x\lambda} & c_2\lambda^{1/2}\ee^{-\lambda^2/2-x\lambda}\epm
%\endeq
\eq
\mathbf{\Psi}(\lambda,x)=\bpm
\lambda^{-\frac{1}{2}}\ee^{\frac{1}{2}\lambda^2+x\lambda} & -\lambda^{-\frac{1}{2}}\ee^{-(\frac{1}{2}\lambda^2+x\lambda)}\\ 0 & \lambda^\frac{1}{2}\ee^{-(\frac{1}{2}\lambda^2 + x\lambda)}
\epm
\mathbf{C},\quad (\lambda,x)\in D\times\mathbb{C},
\label{eq:Hermite-general-solution-Lax-pair}
\endeq
where $\mathbf{C}$ is a matrix independent of both $x$ and $\lambda$.
%$c_1$, $c_2$, $d_1$, and $d_2$ are arbitrary constants. 
\label{lem:gH-direct-general}
\end{lemma} 

The simplest invertible choice for $\mathbf{C}$ is simply $\mathbf{C}=\mathbb{I}$.  Assuming this and also taking the principal branch for $\lambda^\frac{1}{2}$, consider imposing the condition \eqref{eq:infty-asymp} in one of the Stokes sectors near $\lambda=\infty$.  Using $\Theta_\infty=\tfrac{1}{2}$ we find that
\eq
\mathbf{C}=\mathbb{I}\quad\implies\quad \mathbf{\Psi}(\lambda,x)\lambda^{\Theta_\infty\sigma_3}\ee^{-(\frac{1}{2}\lambda^2+x\lambda)\sigma_3} = \bpm 1 & -\lambda^{-1}\\0 & 1\epm
\endeq
which tends to $\mathbb{I}$ as $\lambda\to\infty$ regardless of choice of Stokes sector.  Therefore, there is \emph{no Stokes phenomenon} about the irregular singular point $\lambda=\infty$ in this case, and the normalized solutions associated with the four Stokes sectors are all the same:
\eq
\mathbf{\Psi}^{(\infty)}_j(\lambda,x)=\mathbf{\Psi}^{(\infty)}(\lambda,x):=\bpm
\lambda^{-\frac{1}{2}}\ee^{\frac{1}{2}\lambda^2+x\lambda} & -\lambda^{-\frac{1}{2}}\ee^{-(\frac{1}{2}\lambda^2+x\lambda)}\\ 0 & \lambda^\frac{1}{2}\ee^{-(\frac{1}{2}\lambda^2 + x\lambda)}
\epm,\quad j=1,\dots,4,\quad \arg(\lambda)\in (-\pi,\pi).
\label{eq:gH-matrix-infty}
\endeq

The Fuchsian singular point $\lambda=0$ is resonant in the case $\Theta_0=\tfrac{1}{2}$ under consideration, but the singularity is also apparent as is clear from the absence of logarithms in the general solution \eqref{eq:Hermite-general-solution-Lax-pair}.  Therefore it is possible to choose the matrix $\mathbf{C}=\mathbf{C}^{(0)}$ to define a solution $\mathbf{\Psi}=\mathbf{\Psi}^{(0)}(\lambda,x)$ so that the condition \eqref{eq:zero-asymp} for $\Theta_0=\tfrac{1}{2}$ holds.  This condition requires analyticity of the following expression, which generally has a simple pole at $\lambda=0$:
\eq
\mathbf{\Psi}^{(0)}(\lambda,x)\lambda^{-\Theta_0\sigma_3} = \bpm C^{(0)}_{11}-C^{(0)}_{21} & 0\\0 & 0\epm\lambda^{-1} + \text{holomorphic}.
\endeq
Therefore, the only condition imposed on $\mathbf{C}=\mathbf{C}^{(0)}$ by \eqref{eq:zero-asymp} is that 
$C^{(0)}_{11}=C^{(0)}_{21}$.  Compared with the nonresonant case discussed in Section~\ref{sec:gO-sowing}, the resonant but apparent case here allows for an additional degree of freedom.  It is also convenient to assume that $\det(\mathbf{C}^{(0)})=1$ which in turn guarantees that $\det(\mathbf{\Psi}^{(0)}(\lambda,x))=1$.  Therefore, we will take $\mathbf{C}^{(0)}$ to be a matrix of the form
\eq
\mathbf{C}^{(0)} = \bpm a & b\\a & b+a^{-1}\epm,\quad a\neq 0\quad\implies\quad
\mathbf{\Psi}^{(0)}(\lambda,x)=\bpm
\lambda^{-\frac{1}{2}}\ee^{\frac{1}{2}\lambda^2+x\lambda} & -\lambda^{-\frac{1}{2}}\ee^{-(\frac{1}{2}\lambda^2+x\lambda)}\\ 0 & \lambda^\frac{1}{2}\ee^{-(\frac{1}{2}\lambda^2 + x\lambda)}
\epm\bpm a & b\\a & b+a^{-1}\epm.
\label{eq:gH-matrix-zero}
\endeq
For convenience we pick $a=1$ and $b=0$, which completes the construction of the five canonical simultaneous solutions of the Lax pair for $(\Theta_0,\Theta_\infty)=(\tfrac{1}{2},\tfrac{1}{2})$ and $u(x)=-2x$.  

It is now straightforward to compute the Stokes matrices for the irregular singular point at $\lambda=\infty$ by using \eqref{eq:V21-def}--\eqref{eq:V41-def} and \eqref{eq:gH-matrix-infty}, and the results are trivial except for the contribution of the branch cut of $\lambda^\frac{1}{2}$ along the negative real line, leading directly to \eqref{eq:gH-Stokes} in which
$\Theta_\infty=\tfrac{1}{2}$.  Likewise, we compute the connection matrices using \eqref{eq:V1V3-def}--\eqref{eq:V2V4-def}, \eqref{eq:gH-matrix-infty}, and \eqref{eq:gH-matrix-zero} with $a=1$ and $b=0$ to arrive at \eqref{eq:gH-connection}.

\subsubsection{Reaping the harvest:  use of Schlesinger transformations to span $\Lambda_\mathrm{gH}^{[3]+}$}
The formula \eqref{eq:Y-Psi} clearly gives a solution of Riemann-Hilbert Problem~\ref{rhp:general} for $(\Theta_0,\Theta_\infty)=(\tfrac{1}{2},\tfrac{1}{2})$ when the piecewise constant matrix $\mathbf{V}$ is defined on the jump contour $\Sigma$ with the use of the Stokes matrices \eqref{eq:gH-Stokes} and the connection matrices \eqref{eq:gH-connection}, and this solution returns the rational solution $u(x)=-2x$
of Painlev\'e-IV that started the calculation.  We also observe that, using \eqref{eq:gH-matrix-zero} in \eqref{eq:Y-Psi} and referring to the expansion \eqref{eq:Y-zero-expand}, we have
\eq
\mathbf{Y}^0_0(x)=\bpm 2x & -1\\1 & 0\epm\quad\text{for $(\Theta_0,\Theta_\infty)=(\tfrac{1}{2},\tfrac{1}{2})$.}
\label{eq:Z0-at-gH-seed}
\endeq
The component $\Lambda^{[3]+}_\mathrm{gH}$ of the gH parameter lattice contains the point $(\Theta_0,\Theta_\infty)=(\tfrac{1}{2},\tfrac{1}{2})$ and can be viewed as the set of points $(\Theta_0,\Theta_\infty)$ of the form $(\Theta_0,\Theta_\infty)=(\tfrac{1}{2},\tfrac{1}{2}) + \mathbb{Z}_{\ge 0}(\tfrac{1}{2},\tfrac{1}{2}) + \mathbb{Z}_{\ge 0}(\tfrac{1}{2},-\tfrac{1}{2})$.  
Using Proposition~\ref{prop:Schlesinger}, we see from \eqref{eq:Z0-at-gH-seed} that the only basic Schlesinger transformation that is possibly undefined at $(\Theta_0,\Theta_\infty)=(\tfrac{1}{2},\tfrac{1}{2})\in\Lambda^{[3]+}_\mathrm{gH}$ is $\mathbf{Y}(\lambda;x)\mapsto\mathbf{Y}_\nwarrow(\lambda;x)$, and from Proposition~\ref{prop:Baecklund} we see that of the three that are certainly defined, only $\mathbf{Y}(\lambda;x)\mapsto\mathbf{Y}_\nearrow(\lambda;x)$ and $\mathbf{Y}(\lambda;x)\mapsto\mathbf{Y}_\searrow(\lambda;x)$ are guaranteed to give determinate B\"acklund transformations of $u(x)$.  So, from $(\tfrac{1}{2},\tfrac{1}{2})\in\Lambda_\mathrm{gH}^{[3]+}$ we may certainly step to the nearest neighbor points $(1,1)$ and $(1,0)$ in the same lattice.
Without using specific information such as \eqref{eq:Z0-at-gH-seed}, from Propositions~\ref{prop:nonzero} and \ref{prop:Baecklund}, we observe that
\begin{itemize}
\item The Schlesinger transformation $\mathbf{Y}(\lambda;x)\mapsto\mathbf{Y}_\nearrow(\lambda;x)$ is defined and yields a determinate B\"acklund transformation $u(x)\mapsto u_\nearrow(x)$ at each point of $\Lambda_\mathrm{gH}^{[3]+}$.
\item The Schlesinger transformation $\mathbf{Y}(\lambda;x)\mapsto\mathbf{Y}_\swarrow(\lambda;x)$ is defined at each point of $\Lambda_\mathrm{gH}^{[3]+}$.  Furthermore it is guaranteed to yield a determinate B\"acklund transformation $u(x)\mapsto u_\swarrow(x)$ except possibly at those points of the form $(\Theta_0,\Theta_\infty)=(\tfrac{1}{2},\tfrac{1}{2})+\mathbb{Z}_{\ge 0}(\tfrac{1}{2},-\tfrac{1}{2})$, which satisfy $\Theta_\infty+\Theta_0=1$.
\item The Schlesinger transformation $\mathbf{Y}(\lambda;x)\mapsto\mathbf{Y}_\searrow(\lambda;x)$ is defined at each point of $\Lambda_\mathrm{gH}^{[3]+}$ except possibly at those points of the form $(\Theta_0,\Theta_\infty)=(\tfrac{1}{2},\tfrac{1}{2})+\mathbb{Z}_{\ge 0}(\tfrac{1}{2},\tfrac{1}{2})$ where $\Theta_\infty-\Theta_0=0$.  The corresponding B\"acklund transformation $u(x)\mapsto u_\searrow(x)$ is determinate at every point of $\Lambda_\mathrm{gH}^{[3]+}$, including those points at which the Schlesinger transformation from which it is derived is not guaranteed by Proposition~\ref{prop:nonzero} to be defined.
\item The Schlesinger transformation $\mathbf{Y}(\lambda;x)\mapsto\mathbf{Y}_\nwarrow(\lambda;x)$ is defined at each point of $\Lambda_\mathrm{gH}^{[3]+}$ except possibly at those points of the form $(\Theta_0,\Theta_\infty)=(\tfrac{1}{2},\tfrac{1}{2})+\mathbb{Z}_{\ge 0}(\tfrac{1}{2},\tfrac{1}{2})$ where $\Theta_\infty-\Theta_0=0$.  The corresponding B\"acklund transformation $u(x)\mapsto u_\nwarrow(x)$ is guaranteed to be determinate except possibly at the same excluded points.
\end{itemize}
Therefore, starting from the seed $(\Theta_0,\Theta_\infty)=(\tfrac{1}{2},\tfrac{1}{2})$, we may arrive at an arbitrary point in $\Lambda_\mathrm{gH}^{[3]+}$ by iteratively applying Schlesinger transformations resulting in determinate B\"acklund transformations according to the following principles.
\begin{itemize}
\item To reach a point of the form $(\Theta_0,\Theta_\infty)=(n+1)(\tfrac{1}{2},\tfrac{1}{2})\in\Lambda_\mathrm{gH}^{[3]+}$, $n=0,1,2,\dots$, we apply $\mathbf{Y}(\lambda;x)\mapsto\mathbf{Y}_\nearrow(\lambda;x)$ $n$ times.
\item To reach any other point of $\Lambda_\mathrm{gH}^{[3]+}$, we first step from $(\tfrac{1}{2},\tfrac{1}{2})$ to $(1,0)$ using $\mathbf{Y}(\lambda;x)\mapsto \mathbf{Y}_\searrow(\lambda;x)$.  Then we choose any path in $\Lambda_\mathrm{gH}^{[3]+}$ from $(1,0)$ to the target point that contains only points with $\Theta_\infty<\Theta_0$, and iteratively apply the Schlesinger transformations associated with the steps in the selected path.
\end{itemize}
We have thus arrived at the Riemann-Hilbert representation in Theorem~\ref{thm:gH-RHP} of the gH rational solutions of Painlev\'e-IV for parameters in $\Lambda_\mathrm{gH}^{[3]+}$, and via the non-isomonodromic B\"acklund transformation 
$u(x)\mapsto u_\tw(x)$, in $\Lambda_\mathrm{gH}^{[1]-}$.  The rest of its proof 
is nearly the same as that of Theorem~\ref{thm:gO-RHP} in Section~\ref{sec:RHPgO} except that the Schlesinger transformation steps should be taken to follow the above principles.
Note also that the arguments given in Section~\ref{sec:universal-RHPs} connecting Riemann-Hilbert Problem~\ref{rhp:general} in the gH case with pseudo-orthogonal polynomials explain why Schlesinger transformations do not allow one to escape from the set $\Lambda_\mathrm{gH}^{[3]+}$ into any larger lattice spanned by the same lattice vectors.  This is a special phenomenon of the gH rational solutions, since we have seen that the entire $\mathbb{Z}\times\mathbb{Z}$ gO lattice is accessible by Schlesinger transformations from any given lattice point.

\subsubsection{Monodromy data for other families of gH rational solutions}
To further emphasize the fact that the gH rational solutions for $(\Theta_0,\Theta_\infty)\in\Lambda_\mathrm{gH}^{[1]-}\cup\Lambda_\mathrm{gH}^{[2]-}$ cannot be obtained from those for $(\Theta_0,\Theta_\infty)\in\Lambda_\mathrm{gH}^{[3]+}$ by means of isomonodromic transformations, we give some information about the monodromy data for rational solutions with parameters in $\Lambda_\mathrm{gH}^{[1]-}$.  

Applying the transformation $\mathcal{S}_\tw$ to the seed solution $u(x)=-2x$ for $(\Theta_0,\Theta_\infty)=(\tfrac{1}{2},\tfrac{1}{2})\in\Lambda_\mathrm{gH}^{[3]+}$, we obtain the solution $u(x)=x^{-1}$ for parameters $(\Theta_0,\Theta_\infty)=(-\tfrac{1}{2},\tfrac{3}{2})\in\Lambda_\mathrm{gH}^{[1]-}$.  Without loss of generality taking the corresponding solution of \eqref{eq:yODE} to be $y(x)=x^{-1}\ee^{-x^2}$, and noting that from \eqref{eq:LaxPair-z-define} we get $z(x)=1+\tfrac{1}{2}x^{-2}$, the Lax pair matrix coefficients $\mathbf{\Lambda}_0(x)$, $\mathbf{\Lambda}_1(x)$, and $\mathbf{X}_0(x)$ defined in \eqref{eq:JM-A}--\eqref{eq:JM-U} become
\eq
\mathbf{\Lambda}_0(x)=\bpm x & x^{-1}\ee^{-x^2}\\ \ee^{x^2}(x^{-1}-2x) & -x\epm,\;
\mathbf{\Lambda}_1(x)=\bpm -\tfrac{1}{2}(1+x^{-2}) & -\tfrac{1}{2}x^{-2}\ee^{-x^2}\\\ee^{x^2}(1+\tfrac{1}{2}x^{-2}) & \tfrac{1}{2}(1+x^{-2})\epm,\;
\mathbf{X}_0(x)= \mathbf{\Lambda}_0(x)-x\sigma_3.
\endeq
Following the same strategy we have used twice before, we look first for solutions of $\mathbf{\Psi}_x=\mathbf{X\Psi}$, and some experimentation shows that a particular solution is given in closed form by $\Psi_{1j}=f_1(x,\lambda):=x^{-1}\ee^{-x^2-x\lambda}$ and $\Psi_{2j}=f_2(x,\lambda):=-(2x+2\lambda+x^{-1})\ee^{-x\lambda}$.  Applying reduction of order to obtain the general solution, and then determining the dependence of constants of integration on $\lambda$ so that the other Lax equation $\mathbf{\Psi}_\lambda=\mathbf{\Lambda\Psi}$ holds, we obtain the following analogue of Lemma~\ref{lem:gO-direct-general} and Lemma~\ref{lem:gH-direct-general}.
\begin{lemma}
Fix a simply connected domain $D\subset\mathbb{C}\setminus\{0\}$ and a branch of $\lambda^\frac{1}{2}$ analytic on $D$.  Let $\Theta_0=-\tfrac{1}{2}$ and $\Theta_\infty=\tfrac{3}{2}$, and consider the exact solution $u(x)=x^{-1}$ of the corresponding Painlev\'e-IV equation \eqref{p4}.  If $y(x)=x^{-1}\ee^{-x^2}$, then the Lax pair equations \eqref{eq:PIV-Lax-Pair} are simultaneously solvable for all $(\lambda,x)\in D\times(\mathbb{C}\setminus\{0\})$, and every simultaneous solution matrix has the form
\eq
\mathbf{\Psi}(\lambda,x)=\bpm
\lambda^{\frac{1}{2}}\ee^{-\frac{1}{2}\lambda^2}f_1(x,\lambda)
%\lambda^{\frac{1}{2}}x^{-1}\ee^{-\frac{1}{2}(x+\lambda)^2-\frac{1}{2}x^2} 
& \tfrac{1}{2}\lambda^{-\frac{1}{2}}x^{-1}\ee^{\frac{1}{2}(x+\lambda)^2-\frac{1}{2}x^2}-\lambda^\frac{1}{2}x^{-1}\ee^{-\frac{1}{2}(x+\lambda)^2-\frac{1}{2}x^2}\int_0^{x+\lambda}\ee^{z^2}\,\dd z \\
\lambda^\frac{1}{2}\ee^{-\frac{1}{2}\lambda^2}f_2(x,\lambda)
%-\lambda^\frac{1}{2}(2x+2\lambda+x^{-1})\ee^{-\frac{1}{2}(x+\lambda)^2+\frac{1}{2}x^2} 
& 
-(\lambda^\frac{1}{2}+\tfrac{1}{2}\lambda^{-\frac{1}{2}}x^{-1})\ee^{\frac{1}{2}(x+\lambda)^2+\frac{1}{2}x^2}+(2x+2\lambda+x^{-1})\lambda^\frac{1}{2}\ee^{-\frac{1}{2}(x+\lambda)^2+\frac{1}{2}x^2}\int_0^{x+\lambda}\ee^{z^2}\,\dd z
\epm\mathbf{C},
\endeq
where $\mathbf{C}$ is a matrix independent of both $x$ and $\lambda$.
\label{lem:other}
\end{lemma}
Absence of logarithms again indicates that although the Fuchsian singularity of $\mathbf{\Psi}_\lambda=\mathbf{\Lambda\Psi}$ is resonant, it is also apparent, so it is possible to choose $\mathbf{C}=\mathbf{C}^{(0)}$ to define a solution for $|\lambda|<1$ to satisfy \eqref{eq:zero-asymp}.  For solutions defined to satisfy \eqref{eq:infty-asymp} in the four Stokes sectors we determine $\mathbf{C}=\mathbf{C}^{(\infty)}_j$, $j=1,\dots,4$ by asymptotic analysis of the general solution given in Lemma~\ref{lem:other} for large $\lambda$.  The results are
\eq
\mathbf{C}_2^{(\infty)}=\mathbf{C}_3^{(\infty)}=\bpm \ii\sqrt{\pi} & -\tfrac{1}{2}\\ 2 & 0\epm\quad\text{and}\quad
\mathbf{C}_1^{(\infty)}=\mathbf{C}_4^{(\infty)}=\bpm -\ii\sqrt{\pi} & -\tfrac{1}{2}\\ 2 & 0\epm.
\endeq
Therefore, there is no Stokes phenomenon between sectors $S_2$ and $S_3$ or between $S_4$ and $S_1$ (i.e., upon crossing the imaginary Stokes rays), and the corresponding Stokes matrices $\mathbf{V}_{2,3}=\mathbf{V}_{4,1}=\mathbb{I}$ defined in \eqref{eq:V23-def} and \eqref{eq:V41-def} respectively are trivial.  On the other hand, from \eqref{eq:V21-def} and \eqref{eq:V43-def} we get
\eq
\mathbf{V}_{2,1}=\bpm 1 & 0\\ -4\ii\sqrt{\pi} & 1\epm\quad\text{and}\quad
\mathbf{V}_{4,3} = \bpm -1 & 0\\-4\ii\sqrt{\pi} & -1\epm.
\endeq
Therefore, unlike the gH rationals in the parameter lattice $\Lambda_\mathrm{gH}^{[3]+}$, the rationals in the parameter lattice $\Lambda_\mathrm{gH}^{[1]-}$ have nontrivial Stokes matrices on the real rays.  Similar analysis shows that gH rational solutions in the parameter lattice $\Lambda_\mathrm{gH}^{[2]-}$ have instead nontrivial Stokes phenomenon (only) across the imaginary rays.

%% file: okamoto.bbl
\clearpage

\begin{thebibliography}{99}

\bibitem{BaloghBB:2016}
F. Balogh, M. Bertola, and T. Bothner,
``Hankel determinant approach to generalized Vorob'ev-Yablonski polynomials and their roots,''
\textit{Constr.\@ Approx.\@}
\textbf{44},
417--453,
2016.

\bibitem{BassomCH95}
A. P. Bassom, P. A. Clarkson, and A. C. Hicks,
``B\"acklund transformations and solution hierarchies for the fourth Painlev\'e equation,''
\textit{Stud.\@ Appl.\@ Math.\@}
\textbf{95},
1--71,
1995.

\bibitem{BassomCH:1996}
A. P. Bassom, P. A. Clarkson, and A. C. Hicks,
``On the application of solutions of the fourth Painlev\'e equation to various physically motivated nonlinear partial differential equations,''
\textit{Adv.\@ Diff.\@ Eq.\@}
\textbf{1},
175--198,
1996.

\bibitem{BertolaB:2015}
M. Bertola and T. Bothner,
``Zeros of large degree Vorob'ev-Yablonski polynomials via a Hankel determinant identity,''
\textit{Int.\@ Math.\@ Res.\@ Not.\@ IMRN}
\textbf{2015},
9330--9399,
2015.

\bibitem{BilmanBW:2019}
D. Bilman, R. Buckingham, and D. Wang, ``Large-order asymptotics for multiple-pole solitons of the focusing nonlinear Schr\"odinger equation II:  Far-field behavior,'' \texttt{arXiv:1911.04327}, 2019.

\bibitem{BoitiP80}
M. Boiti and F. Pempinelli, ``Nonlinear Schr\"odinger equation, B\"acklund transformations and Painlev\'e transcendents,'' \textit{Nuovo Cim.\@} \textbf{59B}, 40--58, 1980.

\bibitem{BothnerM:2018}
T. Bothner and P. D. Miller,
``Rational solutions of the Painlev\'e-III equation: large parameter asymptotics,''
\textit{Constr.\@ Approx.\@}
\textbf{51}, 
123--224, 
2020.
DOI: 10.1007/s00365-019-09463-4.

\bibitem{BothnerMS18} T. Bothner, P. D. Miller, and Y. Sheng, 
``Rational solutions of the Painlev\'e-III equation,'' 
\textit{Stud.\@ Appl.\@ Math.\@} 
\textbf{141}, 
626--679, 
2018. 
DOI: 10.1111/sapm.12220.

\bibitem{Buckingham18}
R. Buckingham,
``Large-degree asymptotics of rational Painlev\'e-IV functions associated to generalized Hermite polynomials,'' 
\textit{Int.\@ Math.\@ Res\@. Not.\@ IMRN}
\textbf{2018},
rny172,
2018.

\bibitem{BuckinghamM12}
R. Buckingham and P. D. Miller, 
``The sine-Gordon equation in the semiclassical limit: critical behavior near a separatrix,'' 
\textit{J. Anal.\@ Math.\@} 
\textbf{118}, 
397--492, 
2012.

\bibitem{BuckinghamM13}
R. Buckingham and P. D. Miller,
``The sine-Gordon equation in the semiclassical limit:  dynamics of fluxon condensates,'' 
\textit{Memoirs Amer.\@ Math.\@ Soc.\@} 
\textbf{225},
1--136,
2013.

\bibitem{BuckinghamM:2014}
R. Buckingham and P. D. Miller,
``Large-degree asymptotics of rational Painlev\'e-II functions: noncritical behaviour,''
\textit{Nonlinearity}
\textbf{27},
2489--2577,
2014.

\bibitem{BuckinghamM15} 
R. Buckingham and P. D. Miller,
``Large-degree asymptotics of rational Painlev\'e-II functions: critical behaviour,''
\textit{Nonlinearity}
\textbf{28},
1539--1596,
2015.

\bibitem{ChenF:2006}
Y. Chen and M. Feigin,
``Painlev\'e IV and degenerate Gaussian unitary ensembles,''
\textit{J. Phys.\@ A}
\textbf{39},
12381--12393, 
2006.

\bibitem{Clarkson:2003} P. A. Clarkson,
``The fourth Painlev\'e equation and associated special polynomials,''
\textit{J. Math.\@ Phys.\@}
\textbf{44},
5350--5374,
2003.

\bibitem{Clarkson:2006}
P. A. Clarkson,
``Special polynomials associated with rational solutions of the defocusing nonlinear Schr\"odinger equation and the fourth Painlev\'e equation,''
\textit{European J. Appl.\@ Math.\@}
\textbf{17},
293--322,
2006.

\bibitem{Clarkson:2008}
P. A. Clarkson,
``Rational solutions of the Boussinesq equation,''
\textit{Anal.\@ Appl.\@ (Singap.\@)}
\textbf{6},
349--369,
2008.

\bibitem{Clarkson:2009}
P. A. Clarkson,
``Rational solutions of the classical Boussinesq system,''
\textit{Nonlinear Anal.\@ Real World Appl.\@}
\textbf{10},
3360--3371,
2009.

\bibitem{Clarkson:2009b}
P. A. Clarkson,
``Vortices and polynomials,''
\textit{Stud.\@ Appl.\@ Math.\@}
\textbf{123},
37--62,
2009.

\bibitem{ClarksonT:2009}
P. A. Clarkson and B. Thomas,
``Special polynomials and exact solutions of the dispersive water wave and modified Boussinesq equations,'' 
in \emph{Proceedings of Group Analysis of Differential Equations and Integrable Systems IV}, 
62--76, 
2009.

\bibitem{DaiK:2009}
D. Dai and A. Kuijlaars,
``Painlev\'e IV asymptotics for orthogonal polynomials with respect to a modified Laguerre weight,''
\textit{Stud.\@ Appl.\@ Math.\@}
\textbf{122},
29--83, 
2009.

\bibitem{DeiftVZ94}
P. Deift, S. Venakides, and X. Zhou, ``The collisionless shock region for the long-time behavior of solutions of the KdV equation,'' \textit{Comm.\@ Pure Appl.\@ Math.\@} \textbf{47}, 199--206, 1994.

\bibitem{DeiftZ93}
P. Deift and X. Zhou, ``A steepest descent method for oscillatory Riemann-Hilbert problems:  asymptotics for the mKdV equation,'' \textit{Ann.\@ Math.\@} \textbf{137}, 295--368, 1993.

\bibitem{Dubrovin81}  B. A. Dubrovin, 
``Theta functions and non-linear equations,''
\textit{Russian Math.\@ Surveys} \textbf{36}, 11--92, 1981.

\bibitem{FokasIKN:2006}
A. S. Fokas, A. R. Its, A. A. Kapaev, and V. Yu.\@ Novokshenov,
\textit{Painlev\'e Transcendents.  The Riemann-Hilbert Approach,}
AMS Mathematical Surveys and Mongraphs
\textbf{128}, Amer.\@ Math.\@ Soc., 
Providence, RI.,
2006.

\bibitem{FokasIK:1991}
A. S. Fokas, A. R. Its, and A. V. Kitaev,
``Discrete Painlev\'e equations and their appearance in quantum gravity,''
\textit{Comm.\@ Math.\@ Phys.\@}
\textbf{142},
313--344,
1991.

\bibitem{FokasMA88}
A. S. Fokas, U. Mugan, and M. J. Ablowitz, ``A method of linearization for Painlev\'e equations:  Painlev\'e IV, V,'' \textit{Physica D} \textbf{30}, 247--283, 1988.

\bibitem{ForresterW:2001}
P. Forrester and N. Witte,
``Application of the $\tau$-function theory of Painlev\'e equations to random matrices:  PIV, PII and the GUE,''
\textit{Comm.\@ Math.\@ Phys.\@}
\textbf{219},
357--398,
2001.

\bibitem{Gromak:1987}
V. Gromak,
``On the theory of the fourth Painlev\'e equation,''
\textit{Differentsialnye Uravneniya}
\textbf{23},
760--768,
1987.  (In Russian.)

\bibitem{Jenkins58}
J. A. Jenkins, \textit{Univalent Functions and Conformal Mapping}, Springer, Berlin, 1958.

\bibitem{Jimbo:1981a} M. Jimbo and T. Miwa,
``Monodromy preserving deformation of linear ordinary differential equations with rational coefficients.  II,''
\textit{Physica D}
\textbf{2},
407--448,
1981.

\bibitem{Lukashevich:1967}
N. Lukashevich,
``The theory of Painlev\'e's fourth equation,''
\textit{Differensialnye Uravnenija}
\textbf{3},
771--780,
1967.  (In Russian.)

\bibitem{MarikhinSBP:2000}
V. Marikhin, A. Shabat, M. Boiti, and F. Pempinelli,
``Self-similar solutions of equations of the nonlinear Schr\"odinger type,''
\textit{J. Exp.\@ Theor.\@ Phys.\@} 
\textbf{90},
553--561,
2000.
Translation of 
\textit{Zh.\@ Eksper.\@ Teoret.\@ Fiz.\@}
\textbf{117},
634--643,
2000.
(In Russian.)

\bibitem{MarquetteQ:2016}
I. Marquette and C. Quesne,
``Connection between quantum systems involving the fourth Painlev\'e transcendent and $k$-step rational extensions of the harmonic oscillator related to Hermite exceptional orthogonal polynomial,''
\textit{J. Math.\@ Phys.\@}
\textbf{57},
052101 (15 pp.),
2006.

\bibitem{MasoeroR18}
D. Masoero and P. Roffelsen,
``Poles of Painlev\'e IV rationals and their distribution,''
\textit{SIGMA Symmetry Integrability Geom.\@ Methods Appl.\@}
\textbf{14},
paper no.\@ 002 (49 pp.),
2018.

\bibitem{MasoeroR19}
D. Masoero and P. Roffelsen,
``Roots of generalised Hermite polynomials when both parameters are large,''
\texttt{arXiv:1907.08552v2}, 2019.

\bibitem{MasoeroR20}
D. Masoero and P. Roffelsen,  private communication, 2020.

\bibitem{MillerS:2017}
P. D. Miller and Y. Sheng,
``Rational solutions of the Painlev\'e-II equation revisited,''
\textit{SIGMA Symmetry Integrability Geom.\@ Methods Appl.\@}
\textbf{13},
paper no.\@ 065 (29 pp.),
2017.

\bibitem{Murata:1985}
Y. Murata,
``Rational solutions of the second and the fourth Painlev\'e equations,''
\textit{Funkcial.\@ Ekvac.\@}
\textbf{28},
1--32,
1985.

\bibitem{DLMF} NIST Digital Library of Mathematical Functions.  \texttt{http://dlmf.nist.gov/}, Release 1.0.27 of 2020-06-15.  Online companion to \cite{OLBC10}.

\bibitem{NoumiY:1999}
M. Noumi and Y. Yamada,
``Symmetries in the fourth Painlev\'e equation and Okamoto polynomials,''
\textit{Nagoya Math.\@ J.}
\textbf{153},
53--86,
1999.

\bibitem{NovokshenovS14} V. Yu.\@ Novokshenov and A. A. Shchelkonogov, 
``Double scaling limit in the Painlev\'e IV equation and asymptotics of the Okamoto polynomials,''
\textit{Amer.\@ Math.\@ Soc.\@ Trans.\@} \textbf{233}, 199--210, 2014.

\bibitem{NovokshenovS15} 
V. Yu.\@ Novokshenov and A. A. Shchelkonogov, ``Distribution of zeroes to generalized Hermite polynomials,'' \textit{Ufa Math.\@ J.} \textbf{7}, 54--66, 2015.

\bibitem{Okamoto:1986}
K. Okamoto,
``Studies on the Painlev\'e equations III.  Second and fourth Painlev\'e equations, P$_{\rm II}$ and P$_{\rm IV}$,''
\textit{Math.\@ Ann.\@}
\textbf{275},
221--255,
1986.

\bibitem{OLBC10} F. W. J. Olver, D. W. Lozier. R. F. Boisvert, and C. W. Clark, editors.  \textit{NIST Handbook of Mathematical Functions}, Cambridge University Press, New York, NY, 2010.  Print companion to \cite{DLMF}.

\bibitem{OsipovSZ:2010}
V. Osipov, H. Sommers, and K. Zyczkowski,
``Random Bures mixed states and the distribution of their purity,''
\textit{J. Phys.\@ A}
\textbf{43},
055302 (22 pp.),
2010.

\bibitem{Strebel84} K. Strebel, \textit{Quadratic Differentials}, Springer Verlag, Berlin, 1984.

\bibitem{VanAssche18} W. Van Assche, \textit{Orthogonal Polynomials and Painlev\'e Equations}, Australian Mathematical Society Lecture Series vol.\@ 27, Cambridge University Press, Cambridge, United Kingdom, 2018.


\begin{comment}

\bibitem{Kajiwara:1998} K. Kajiwara and Y. Ohta,
Determinant structure of the rational solutions for the Painlev\'e IV equation,
\textit{J. Phys. A}
\textbf{31},
2431--2446
(1998).

\bibitem{Murata:1985}
Y. Murata,
Rational solutions of the second and the fourth Painlev\'e equations,
\textit{Funkcial. Ekvac.}
\textbf{28},
1--32
(1985).

\bibitem{Noumi:1999} M. Noumi and Y. Yamada,
Symmetries in the fourth Painlev\'e equation and Okomoto polynomials,
\textit{Nagoya Math. J.}
\textbf{153},
53--86
(1999).

\bibitem{BaloghBB:2016}
F. Balogh, M. Bertola, and T. Bothner,
Hankel determinant approach to generalized Vorob'ev-Yablonski polynomials and their roots,
\textit{Constr. Approx.}
\textbf{44},
417--453
(2016).

\bibitem{Bertola:2014} 
M. Bertola and T. Bothner,
Zeros of large degree Vorob'ev-Yablonski polynomials via a Hankel determinant identity,
\textit{Int. Math. Res. Not. IMRN}
\textbf{2015},
9330--9399
(2015).

\bibitem{Buckingham:2012} 
R. Buckingham and P. Miller,
The sine-Gordon equation in the semiclassical limit:  critical behavior near a separatrix,
\textit{J. Anal. Math.}
\textbf{118},
397--492
(2012).

\bibitem{Buckingham:2013} 
R. Buckingham and P. Miller,
Large-degree asymptotics of rational Painlev\'e-II functions: noncritial behaviour,
\textit{Nonlinearity}
\textbf{27},
2489--2577
(2014).


\bibitem{ChenF:2006}
Y. Chen and M. Feigin,
Painlev\'e IV and degenerate Gaussian unitary ensembles,
\textit{J. Phys. A}
\textbf{39},
12381--12393
(2006).

\bibitem{Clarkson:2003b} P. Clarkson,
The third Painlev\'e equation and associated special polynomials,
\textit{J. Phys. A}
\textbf{36},
9507--9532
(2003).

\bibitem{Clarkson:2005} P. Clarkson,
Special polynomials associated with rational solutions of the fifth Painlev\'e equation,
\textit{J. Comput. Appl. Math.}
\textbf{178},
111--129
(2005).

\bibitem{Clarkson:2006} P. Clarkson,
Special polynomials associated with rational solutions of the defocusing nonlinear Schr\"odinger equation and the fourth Painlev\'e equation,
\textit{European J. Appl. Math.}
\textbf{3},
293--322 
(2006).

\bibitem{Clarkson:2008} P. Clarkson,
Rational solutions of the Boussinesq equation,
\textit{Anal. Appl. (Singap.)}
\textbf{6},
349--369
(2008).

\bibitem{Clarkson:2009a} P. Clarkson,
Rational solutons of the classical Boussinesq system,
\textit{Nonlinear Anal. Real World Appl.}
\textbf{10},
3360--3371
(2009).

\bibitem{Clarkson:2009b} P. Clarkson,
Vortices and polynomials,
\textit{Stud. Appl. Math.}
\textbf{123},
37--62
(2009).

\bibitem{Clarkson-Mansfield:2003} P. Clarkson and E. Mansfield,
The second Painlev\'e equation, its hierarchy and associated special polynomials,
\textit{Nonlinearity}
\textbf{16},
R1--R26
(2003).

\bibitem{Clarkson:2009c} P. Clarkson and B. Thomas,
Special polynomials and exact solutions of the dispersive water wave and 
modified Boussinesq equations,
in 
\textit{Proceedings of Group Analysis of Differential Equations and Integrable Systems IV},
62--76
(2009).

\bibitem{DeiftKMVZ:1999}
P. Deift, T. Kriecherbauer, K. McLaughlin, S. Venakides, and X. Zhou,
Uniform asymptotics for polynomials orthogonal with respect to varying 
exponential weights and applications to universality questions in random 
matrix theory,
\textit{Comm. Pure Appl. Math.}
\textbf{52},
1335--1425
(1999).

\bibitem{DeiftZ:1993}
P. Deift and X. Zhou,
A steepest descent method for oscillatory Riemann-Hilbert problems. Asymptotics for the MKdV equation,
\textit{Ann. of Math. (2)}
\textbf{137},
295--368
(1993).

\bibitem{FelderHV:2012}
G. Felder, A. Hemery, and A. Veselov, 
Zeros of Wronskians of Hermite polynomials and Young diagrams,
\textit{Physica D}
\textbf{241},
2131--2137
(2012).

\bibitem{Filipuk:2008} G.Filipuk and P. Clarkson,
The symmetric fourth Painlev\'e hierarchy and associated special polynomials,
\textit{Stud. Appl. Math.}
\textbf{121},
157--188
(2008).

\bibitem{Fokas:1991} A. Fokas, A. Its, and A. Kitaev,
Discrete Painlev\'e equations and their appearance in quantum gravity,
\textit{Comm. Math. Phys.}
\textbf{142},
313--344
(1991).

\bibitem{Ismail:2005}
M. Ismail,
Determinants with orthogonal polynomial entries.
\textit{J. Comput. Appl. Math.} 
\textbf{178} 
(2005), 
255--266.

\bibitem{KarlinS:1961}
S. Karlin and G. Szeg\"o,
On certain determinants whose elements are orthogonal polynomials,
\textit{J. Anal. Math.}
{\bf 8}, 
1--157
(1961).

\bibitem{Krattenthaler:1999}
C. Krattenthaler, 
Advanced determinant calculus.
\textit{Sém. Lothar. Combin.} 
{\bf 42} 
(1999), 
Art. B42q.

\bibitem{KudryashovD:2007}
N. Kudryashov and M. Demina,
Relations between zeros of special polynomials associated with the Painlev\'e 
equations,
\textit{Phys. Lett. A}
\textbf{368},
227--234
(2007).

\bibitem{Lukashevich:1967} N. Lukashevich,
The theory of Painlev\'e's fourth equation (Russian),
\textit{Differencial'nye Uravnenija}
\textbf{3},
771--780
(1967).

\bibitem{Marikhin:2001} V. Marikhin,
Representation of a Coulomb gas for rational solutions of the Painlev\'e 
equations,
\textit{Theoret. and Math. Phys.}
\textbf{127},
646--663
(2001).
Translated from 
\textit{Teoret. Mat. Fiz.}
\textbf{127},
248--303
(2001).

\bibitem{Marquette:2013} I. Marquette and C. Quesne,
Two-step rational extensions of the harmonic oscillator:  exceptional orthogonal polynomials and ladder operators,
\textit{J. Phys. A}
\textbf{46}
155201 
(2013).

\bibitem{MarquetteQ:2016}
I. Marquette and C. Quesne,
Connection between quantum systems involving the fourth Painlev\'e transcendent and k-step rational extensions of the harmonic oscillator related to Hermite exceptional orthogonal polynomial,
\textit{J. Math. Phys.}
\textbf{57},
052101
(2016).


\bibitem{MillerS:2017}
P. Miller and Y. Sheng,
Rational solutions of the Painlev\'e-II equation revisited,
arXiv:1704.04851 (2017).

\bibitem{Novokshenov:2015}
V. Novokshenov and A. Schelkonogov,
Distribution of zeroes to generalized Hermite polynomials,
\textit{Ufa Math. J.}
\textbf{7},
54--66
(2015).  


\bibitem{Bertola:2009} M. Bertola and S. Lee,
``First colonization of a spectral outpost in random matrix theory,''
\textit{Constr. Approx.}
\textbf{30},
225--263, 
2009.

\bibitem{Bertola:2010} M. Bertola and A. Tovbis,
``Universality for the focusing nonlinear Schr\"odinger equation at the gradient catastrophe point: Rational breathers and poles of the tritronqu\'ee solution to Painlev\'e I,''
\texttt{arXiv:1004.1828}, 
2010.

\bibitem{Claeys:2010} T. Claeys and T. Grava,
``Solitonic asymptotics for the Korteweg-de Vries equation in the small dispersion limit,''
\textit{SIAM J. Math. Anal.}
\textbf{42}, 
2132--2154,
2010.

\bibitem{Dubrovin:2009} B. Dubrovin, T. Grava, and C. Klein,
``On universality of critical behavior in the focusing nonlinear Schr\"odinger equation, elliptic umbilic catastrophe and the tritronqu\'ee solution to the Painlev\'e I equation,''
\textit{J. Nonlinear Sci.}
\textbf{19},
57--94,
2009.

\bibitem{Kapaev:2004} A. Kapaev,
``Quasi-linear stokes phenomenon for the Painlev\'e first equation.  
\textit{J. Phys. A}
\textbf{37},
11149--11167,
2004.

\bibitem{Airault:1979} H. Airault,
``Rational solutions of Painlev\'e equations,''
\textit{Stud. Appl. Math.}
\textbf{61},
31--53,
1979.

%\bibitem{BealsDT88} R. Beals, P. Deift, and C. Tomei, \textit{Direct and inverse scattering on the line}, Mathematical Surveys and Monographs vol. 28, American Mathematical Society, Providence, RI, 1988.

\bibitem{BealsC84} R. Beals and R. Coifman, ``Scattering and inverse scattering for first order systems,'' \textit{Comm. Pure Appl. Math.} \textbf{37}, 39--90, 1984.

\bibitem{Boutroux13} P. Boutroux, ``Recherches sur les transcendantes de M.~Painlev\'e et l'\'etude asymptotique des \'equations diff\'erentielles du second ordre,'' \textit{Ann. Sci. \'Ecole Norm. Sup. (3)} 
\textbf{30}, 255--375, 1913; \textbf{31}, 99--159, 1914.

\bibitem{BuckinghamMcritical} R. J. Buckingham and P. D. Miller, 
``The sine-Gordon equation in the semiclassical limit:  critical behavior near a separatrix," 
\textit{J. Anal. Math.}
\textbf{118},
397--492,
2012.

\bibitem{BuckinghamMMemoir} R. J. Buckingham and P. D. Miller, ``The sine-Gordon equation in
the semiclassical limit:  dynamics of fluxon condensates,'' \textit{Memoirs AMS} \textbf{225}, number 1059, 2013.


\bibitem{Clarkson:2009} P. Clarkson,
``Vortices and polynomials,''
\textit{Stud. Appl. Math.}
\textbf{123},
37--62,
2009.

\bibitem{Clarkson:2003} P. Clarkson and E. Mansfield,
``The second Painlev\'e equation, its hierarchy, and associated special polynomials,''
\textit{Nonlinearity}
\textbf{16},
R1--R26,
2003.

\bibitem{Coifman:1982} R. Coifman, A. McIntosh, and Y. Meyer,
``L'int\'egrale de Cauchy d\'efinit un op\'erateur born\'e sur $L^2$ pur les courbes lipschitziennes'' [The Cauchy integral defines a bounded operator on $L^2$ for Lipschitz curves],
\textit{Ann. Math.}
\textbf{116},
361--387,
1982.

\bibitem{Deift:1997} P. Deift, S. Venakides, and X. Zhou, 
``New results in small dispersion KdV by an extension of the steepest descent method for Riemann-Hilbert problems,'' 
\textit{Int. Math. Res. Not. IMRN} 
\textbf{1997}, 
285--299, 
1997.

\bibitem{Deift:1993} P. Deift and X. Zhou, 
``A steepest-descent method for oscillatory Riemann-Hilbert problems:  asymptotics for the mKdV equation,''
\textit{Ann. Math.}
\textbf{137}, 
295--368, 
1993.

\bibitem{Deift:1995} P. Deift and X. Zhou,
``Asymptotics for the Painlev\'e II equation,''
\textit{Comm. Pure Appl. Math.}
\textbf{48},
277--337,
1995.

\bibitem{Fokas:2006-book} A. S. Fokas, A. R. Its, A. A. Kapaev, and V. Yu. 
Novokshenov,
\textit{Painlev\'e Transcendents:  The Riemann-Hilbert Approach},
Volume 128,
Mathematical Surveys and Monographs,
American Mathematical Society,
Providence,
2006.

\bibitem{Flaschka:1980} H. Flaschka and H. Newell,
``Monodromy- and spectrum-preserving deformations. I,''
\textit{Comm. Math. Phys.} 
\textbf{76},
65--116, 
1980.


\bibitem{Jimbo:1981b} M. Jimbo and T. Miwa,
``Monodromy preserving deformation of linear ordinary differential equations with rational coefficients.  III,''
\textit{Physica D}
\textbf{4},
26--46,
1981.

\bibitem{Johnson:2006} C. Johnson,
``String theory without branes,''
\texttt{arXiv:hep-th/0610223}, 
2006.

\bibitem{Kametaka:1983} Y. Kametaka,
``On poles of the rational solution of the Toda equation of Painlev\'e-II type,''
\textit{Proc. Japan Acad. Ser. A Math. Sci.} 
\textbf{59}, 
358--360,
1983.

\bibitem{Kapaev:1997} A. Kapaev,
``Scaling limits in the second Painlev\'e transcendent,''
\textit{J. Math. Sci.}
\textbf{83},
38--61,
1997.  
Translated from 
\textit{Zap. Nauchn. Sem. S.-Peterburg. Otdel. Mat. Inst. Steklov. (POMI)}
\textbf{209},
60--101,
1994.

\bibitem{Muskhelishvili:1992} N. Muskhelishvili,
\textit{Singular Integral Equations},
Second Edition,
Dover,
1992.

\bibitem{Murata:1985} Y. Murata,
``Rational solutions of the second and fourth Painlev\'e equations,''
\textit{Funkcialaj Ekvacioj}
\textbf{28},
1--32,
1985.

\bibitem{Roffelson:2010} P. Roffelson,
``Irrationality of the roots of the Yablonskii-Vorob'ev polynomials and relations between them,''
\textit{ SIGMA Symmetry Integrability Geom. Methods Appl.}
\textbf{6},
1--11,
2010.

\bibitem{Roffelson:2012} P. Roffelson,
``On the number of real roots of the Yablonskii-Vorob'ev polynomials,''
\textit{ SIGMA Symmetry Integrability Geom. Methods Appl.}
\textbf{8},
1--9,
2012.

\end{comment}



\end{thebibliography}
